\documentclass[12pt]{book}

\usepackage[colorlinks=true,linkcolor=blue,bookmarks=true]{hyperref}
\usepackage{fancyhdr}
\usepackage{makeidx}
\usepackage{amsmath,amssymb,dsfont}
\usepackage{graphicx,color}
\usepackage{enumitem, xspace, float}
\usepackage[usenames,dvipsnames]{xcolor}
\usepackage{graphicx}
\usepackage{dsfont}
\usepackage{cases}

\makeindex
\def\abstract#1{{\noindent\textbf{Abstract.} \em #1}
\par
\medskip

}

\definecolor{Labo}{rgb}{0.9,0.9,1.0}
\definecolor{Black}{rgb}{0.00,0.00,0.00}
\definecolor{Red}{rgb}{1.00,0.00,0.00}
\definecolor{Blue}{rgb}{0.00,0.00,1.00}
\definecolor{White}{rgb}{1.00,1.00,1.00}
\definecolor{aqua}{rgb}{0.00,255,255}
\definecolor{Yellow}{rgb}{1.00,0.00,0.00}

\newtheorem{theorem}{Theorem}[section]
\newtheorem{lemma}[theorem]{Lemma}
\newtheorem{corollary}[theorem]{Corollary}
\newtheorem{proposition}[theorem]{Proposition}
\newtheorem{definition}[theorem]{Definition}
\newtheorem{remark}[theorem]{Remark}

\newcounter{example}
\counterwithin{example}{chapter}
    \newenvironment{example}[1][]{\refstepcounter{example}\par\medskip\noindent
    \textbf{Example~\theexample #1 ---} \rmfamily}{\medskip}

\newenvironment{proof}[1][Proof]{\medskip\noindent\emph{#1 ---\;}}{\\ 
    \null\  \hfill\textbf{Q.E.D.}\medskip}

\DeclareMathOperator{\dist}{\mathrm{dist}}

\def\e{\varepsilon}
\def\eps{\varepsilon}

\newcommand{\lsc}{l.s.c.\xspace}
\newcommand{\usc}{u.s.c.\xspace}

\newcommand{\cob}{\overline{\mathop{\rm co}}}

\def\be{\begin{equation}}
\def\ee{\end{equation}}

\def\R{\mathbb R}
\def\N{\mathbb N}
\def\Z{\mathbb Z}
\def\F{\mathbb{F}}
\def\Fk{\mathbb{F}^k}

\def\G{\mathbb G}
\def\I{\mathbb I}

\def\OOb{\overline{\mathcal{O}}}
\def\Omegb{{\overline \Omega}}
\def\domeg{{\partial \Omega}}
\renewcommand{\H}{\mathcal{H}}
\newcommand{\Ga}{\Gamma}

\newcommand{\V}[1]{V^{(#1)}}

\def\OO{\mathcal{O}}

\def\VV{\mathcal{V}}

\renewcommand{\H}{\mathcal{H}}

\def\KK{{\mathcal K}}
\def\VV{{\mathcal V}}

\newcommand{\mT}{\mathcal{T}}
\newcommand{\mTreg}{\mathcal{T}^\mathrm{reg}}

\newcommand{\E}{\mathcal{E}}
\newcommand{\cT}{\mathcal{T}}

\def\ue{u_\eps}

\def\xe{x_\eps}
\def\pe{p_\eps}
\def\te{t_\eps}
\def\ye{y_\eps}
\def\re{r_\eps}
\def\se{s_\eps}
\def\ze{z_\eps}

\def\ou{\overline{u}}
\def\uu{\underline{u}}

\def\ow{\overline{w}}

\def\ovp{\overline{p}}
\def\unp{\underline{p}}
\def\ovq{\overline{q}}
\def\unq{\underline{q}}
\def\xb{{\bar x}}
\def\yb{{\bar y}}
\def\zb{{\bar z}}
\def\tb{{\bar t}}
\def\ovs{{\bar s}}

\def\tb{{\bar t}}

\def\gb{{\mathbf b}}
\def\gc{{\mathbf c}}
\def\gl{{\mathbf l}}

\def\uc{{\underline{c}}}

\def\vcb{{\bar X}}
\def\vc{X}
\def\vt{Y}
\def\vtb{\bar Y}
\def\vnb{\bar Z}
\def\vte{Y_\e}

\def\vn{Z}

\def\tX{{\tilde X}}
\def\tr{{\tilde r}}

\def\bx{{\bar{x}}}
\def\bt{{\bar{t}}}

\def\Um{\mathbf{U}^-}
\def\Up{\mathbf{U}^+}
\def\VF{\mathbf{U}}
\def\VFp{\mathbf{U}^+}
\def\VFm{\mathbf{U}^-}

\newcommand{\Uim}{\mathbf{U}^\mathrm{FL}}

\def\u0{u_0}

\def\Ge{\G_\eps}

\newcommand{\Hti}{\tilde{H}}
\newcommand{\Htireg}{\tilde H^{\rm reg}}

\def\HT{\HT}

\def\Hb{\bar{H}}

\def\HTreg{{ H}^{\rm reg}_T}

\def\FTreg{{F}^{\rm reg}_T}
\def\FT{{F}_T}

\def\limssup{\mathop{\rm limsup\hspace*{0.1em}\raisebox{0.25em}{$\scriptstyle \ast$}\,}}
\def\limiinf{\mathop{\rm liminf\hspace*{0.07em}\raisebox{-0.15em}{$\scriptstyle \ast$}\,}}

\def\YtoXandEPStoZERO{{\displaystyle{\mathop{\scriptstyle{Y\to
X}}_{\varepsilon\to 0}}}}

\renewcommand{\d}{\,\mathrm{d}}

\newcommand{\dalpha}{\,\mathrm{d}\alpha}
\newcommand{\ds}{\,\mathrm{d}s}
\newcommand{\dt}{\,\mathrm{d}t}

\def\apg{\left\{}
\def\chg{\right\}}

\def\ch{\right.}

\def\a{a}

\def\A{{\cal A}}

\def\Ta{{\cal T}_{x,t}}

\def\Treg{{\cal T}^{\rm reg}_{x,t}}

\def\HTreg{{H}^{\rm reg}_T}

\def\HT{{H}_T}

\def\Aoreg{A_0^{\rm reg}}

\newcommand\PP[1]{$\mathbf{P}(#1)$}

\newcommand{\Man}[1]{\mathbf{M}^{#1}}
\newcommand{\tMan}[1]{\mathbf{\tilde M}^{#1}}

\newcommand{\hyp}[1]{$(\mathbf{H}_\mathbf{#1})$}

\newcommand{\BCL}{\mathbf{BCL}}

\newcommand{\hU}{\mathbf{U}}
\newcommand{\Sub}{\mathbf{Sub}}

\newcommand{\B}{\mathbf{B}}
\newcommand{\C}{\mathbf{C}}
\renewcommand{\L}{\mathbf{L}}

\newcommand{\LCR}{{\bf (LCR)}\xspace}
\newcommand{\SCR}{{\bf (SCR)}\xspace}
\newcommand{\GCR}{{\bf (GCR)}\xspace}
\newcommand{\LCRxt}{$\mathbf{LCR}^\psi(\xb,\tb)$\xspace}

\newcommand{\LCREV}{\text{{\bf (LCR)}-\emph{evol}}\xspace}
\newcommand{\GCREV}{\text{{\bf (GCR)}-\emph{evol}}\xspace}

\newcommand{\HJgen}{{\rm (HJ-Gen)}\xspace}
\newcommand{\KC}{{\rm (KC)}\xspace}

\newcommand{\FL}{{\rm (FL)}\xspace}
\newcommand{\GFL}{{\rm (GFL)}\xspace}
\newcommand{\GJC}{{\rm (GJC)}\xspace}
\newcommand{\CVS}{{\rm (CVS)}\xspace}
\newcommand{\JVS}{{\rm (JVS)}\xspace}
\newcommand{\JVSub}{{\rm (JVSub)}\xspace}
\newcommand{\JVSup}{{\rm (JVSuper)}\xspace}
\newcommand{\FLS}{{\rm (FLS)}\xspace}
\newcommand{\FLSub}{{\rm (FLSub)}\xspace}
\newcommand{\FLSup}{{\rm (FLSuper)}\xspace}

\newcommand{\wSSub}{{\rm (w-S-Sub)}\xspace}
\newcommand{\sSSub}{{\rm (s-S-Sub)}\xspace}
\newcommand{\SSup}{{\rm (S-Super)}\xspace}

\newcommand{\SBJ}{{\rm (SBJ)}\xspace}

\newcommand{\HJBS}{\text{\rm (HJB-S)}\xspace}

\newcommand{\TC}{{\bf (TC)}\xspace}
\newcommand{\TCs}{{\bf (TC-$s$)}\xspace}

\newcommand{\TCsH}{{\bf (TC-HJ)}\xspace}

\newcommand{\NCe}{{\bf (NC)}\xspace}
\newcommand{\NCw}{{\bf (NC$_w$)}\xspace}

\newcommand{\NCHJ}{{\bf (NC-HJ)}\xspace}

\newcommand{\Mong}{{\bf (Mon)}\xspace}

\newcommand{\TCBCL}{{\bf (TC-BCL)}\xspace}
\newcommand{\NCBCL}{{\bf (NC-BCL)}\xspace}
\newcommand{\M}{\mathbb{M}}
\newcommand{\tM}{\mathbb{\tilde M}}

\newcommand{\GAConv}{{\bf (GA-Conv)}\xspace}

\newcommand{\GACC}{{\bf (GA-CC)}\xspace}
\newcommand{\GAQC}{{\bf (GA-QC)}\xspace}
\newcommand{\GAGen}{{\bf (GA-Gen)}\xspace}

\newcommand{\HContG}{{\bf (GA-ContG)}\xspace}
\newcommand{\GAGFL}{{\bf (GA-G-FL)}\xspace}
\newcommand{\GAGKT}{{\bf (GA-G-GKT)}\xspace}
\newcommand{\GAGFLT}{{\bf (GA-G-FLT)}\xspace}

\newcommand{\AFS}{{\rm (AFS)}\xspace}
\newcommand{\LFS}{\text{\rm (LFS)}\xspace}
\newcommand{\TFS}{{\rm (TFS)}\xspace}

\newcommand{\Sbb}{\mathbb{S}}

\newcommand{\SSP}{{\rm (SSP)}\xspace}

\newcommand{\QRB}{{\bf (QRB)}\xspace}

\newcommand{\dN}{\mathrm{d}_{\mathbf{N}}}

\newcommand{\toLFSs}{\mathop{\xrightarrow{\mathrm{LFS-}s\,}}}

\newcommand{\toTFSs}{\mathop{\xrightarrow{s\,}}}
\newcommand{\toTFSw}{\mathop{\xrightarrow{w\,}}}

\def\Esing{E_{\rm sing}^{\eta}}
\def\sym{{\cal S}^N}

\newcommand{\ind}[1]{{1\hspace{-1.2mm}{\rm I}}_{\{#1\}}}
\def\1{{1\hspace{-1.2mm}{\rm I}}}

\newcommand{\commentout}[1]{}

\newcommand{\resp}[1]{$[\,resp.$ #1$\,]$\xspace}

\newcommand{\mF}{\mathcal{F}}
\newcommand{\mA}{\mathcal{A}}

\def\Tf{T_f}
\def\tTf{\tilde{T}_f}

\def\PC1{\mathrm{PC}^{1}(\R^N\times [0,\Tf])}

\newcommand{\cyl}{Q^{x,t}_{r,h}}
\newcommand{\cylp}{Q^{x,t}_{r',h'}}
\newcommand{\cylb}{\overline{Q^{x,t}_{r,h}}}

\newcommand{\Monu}{$(\mathbf{Mon}\text{-}u)$\xspace}
\newcommand{\Monp}{$(\mathbf{Mon}\text{-}p)$\xspace}
\newcommand{\reg}{{\rm reg}}
\def\xoe{{\displaystyle \frac{x}{\varepsilon}}}

\newcommand{\HHJ}{$(\mathbf{H}_\mathbf{BA-HJ})$\xspace}
\newcommand{\HCP}{$(\mathbf{H}_\mathbf{BA-CP})$\xspace}
\newcommand{\HGCP}{$(\mathbf{H}_\mathbf{BA-}p_t)$\xspace}
\newcommand{\HBACP}{$(\mathbf{H}_\mathbf{BA-CP})$\xspace}
\newcommand{\HBAHJ}{$(\mathbf{H}_\mathbf{BA-HJ})$\xspace}
\newcommand{\HBACPp}{$(\mathbf{H}_\mathbf{BA-CP})^+$\xspace}
\newcommand{\HBAHJp}{$(\mathbf{H}_\mathbf{BA-HJ})^+$\xspace}
\newcommand{\HBAIDCP}{$(\mathbf{H}_\mathbf{BA-ID}^\mathbf{CP})$\xspace}

\newcommand{\HBACPc}{$(\mathbf{H}_\mathbf{BA-CP}^\mathrm{class.})$\xspace}
\newcommand{\HBAConv}{$(\mathbf{H}_\mathbf{BA-Conv})$\xspace}
\newcommand{\HSubHJ}{$(\mathbf{H}_\mathbf{Sub-HJ})$\xspace}
\newcommand{\HConv}{$(\mathbf{H}_\mathbf{Conv})$\xspace}

\newcommand{\NCoH}{$(\mathbf{NC}_\mathbf{\H})$\xspace}
\newcommand{\HQC}{$(\mathbf{H}_\mathbf{QC})$\xspace}
\newcommand{\HST}[1][]{$(\mathbf{H}_\mathbf{ST}^{#1})$\xspace}
\newcommand{\HBCLa}{$(\mathbf{H}_\mathbf{BCL})_{fund}$\xspace}
\newcommand{\HBCLb}{$(\mathbf{H}_\mathbf{BCL})_{struct}$\xspace}
\newcommand{\HBCL}{$(\mathbf{H}_\mathbf{BCL})$\xspace}
\newcommand{\HBASF}{$(\mathbf{H}_\mathbf{BA-SF})$\xspace}
\newcommand{\HBASFstar}{$(\mathbf{H}_\mathbf{BA-SF}^\mathbf{*})$\xspace}
\newcommand{\HSBJ}{$(\mathbf{H}_\mathbf{SBJ})$\xspace}
\newcommand{\Hgamma}{$(\mathbf{H}_{\gamma,g})$\xspace}
\newcommand{\Hgam}[1][]{$(\mathbf{H}_{\gamma,g}^{#1})$\xspace}

\newcommand{\HSBC}{$(\mathbf{H}_\mathbf{BC}^\emph{simpl.})$\xspace}

\newcommand{\LOCa}{$(\mathbf{LOC1})$\xspace}
\newcommand{\LOCb}{$(\mathbf{LOC2})$\xspace}
\newcommand{\USCS}{\textrm{\rm USC-Sub}\xspace}
\newcommand{\LSCS}{\textrm{\rm LSC-Sup}\xspace}

\newcommand{\LOCaEV}{$(\mathbf{LOC1})$-\emph{evol}\xspace}
\newcommand{\LOCbEV}{$(\mathbf{LOC2})$-\emph{evol}\xspace}

\newcommand{\floorF}[1]{{\lfloor #1 \rfloor}^x_{\partial\mF^{x,r}}}
\newcommand{\floorFt}[1]{{\lfloor #1 \rfloor}^{(x,t)}_{\partial_\mathrm{\,lat} Q}}

\newcommand{\ved}{v^{\e,\delta}}
\newcommand{\Se}{{S_\e}}

\newcommand{\epsn}{{\eps_n}}

\newcommand{\BCLHO}{\BCL^{H_0}}
\newcommand{\VFmHO}{\VFm_{H_0}}
\newcommand{\VFpHO}{\VFp_{H_0}}
\newcommand{\AHO}{A^{H_0}}
\newcommand{\mAHO}{\mA^{H_0}}
\newcommand{\HHO}{H^{H_0}}
\newcommand{\FHO}{\F^{H_0}}
\newcommand{\mTHO}{\mT_{H_0}}
\newcommand{\cO}{\mathcal{O}}

\newcommand{\In}{{C}^+}
\newcommand{\Inf}{{C}^+_\mathrm{flat}}

\newcommand{\Kxt}{\mathcal{K}_{(\xb,\tb)}}
\newcommand{\BCLloc}{\BCL^{\delta}_\mathrm{loc}}
\newcommand{\mK}{\mathcal{K}}

\newcommand{\mC}{\mathcal{C}}
\newcommand{\QCR}{$(\mathbf{H}_{\mathbf{QC-}\R})$\xspace}

\newcommand{\QOm}{Q^{\bar\Omega}_r}

\newcommand{\BCLp}{\BCL'}

\newcommand{\IDP}{$\mathbf{(IDP)}$\xspace}
\newcommand{\IDPN}{$\mathbf{(IDPN)}$\xspace}

\newcommand{\BCLeq}{\BCL_\mathrm{eq}}
\newcommand{\BCLbc}{\BCL_\mathrm{bc}}

\newcommand{\mus}{{\mu_\star}}

\newcommand{\wlg}{\emph{w.l.o.g.}\xspace}
\newcommand{\ie}{\emph{i.e.}\xspace}
\newcommand{\wrt}{\emph{w.r.t.}\xspace}
\newcommand{\cf}{\emph{cf.}\xspace}
\newcommand{\adhoc}{\emph{ad hoc}\xspace}

\usepackage{adjustbox}

\newcommand{\smsp}{\ \\[3pt]}

\newenvironment{assumption}[2]{%
    \vspace{1em}
    \noindent #1\ --- {#2}\nopagebreak\smsp
    \begingroup\it}
    {\endgroup\vspace{1em}
    }

\newenvironment{app}[3]{%
    \vspace*{1.5em}
    \noindent #1\ --- \emph{#2}~{#3:}\nopagebreak}
    {}

\newcommand{\keypoint}[2]{%
    \begin{center}
    \begin{minipage}{0.8\textwidth}
    \textsc{#1 ---}
    \emph{#2} 
    \end{minipage}
    \end{center}
}

\newcommand{\HSTLFS}{$(\mathbf{H}_{\mathbf{ST}}^\text{\sc lfs})$\xspace}
\newcommand{\HSTTFS}{$(\mathbf{H}_{\mathbf{ST}}^\text{\sc tfs})$\xspace}
\newcommand{\HSTFLAT}{$(\mathbf{H}_{\mathbf{ST}}^\text{\sc flat})$\xspace}
\newcommand{\HSTGEN}{$(\mathbf{H}_{\mathbf{ST}}^\text{\sc gen})$\xspace}

\begin{document}

\pagestyle{empty}
\thispagestyle{empty}
\pagestyle{fancy}
\fancyhf{}

%
%

\title{An Illustrated Guide of the Modern Approaches of Hamilton-Jacobi Equations and Control
Problems with Discontinuities} 
\author{Guy Barles\footnote{\texttt{<guy.barles@idpoisson.fr>}} \&
Emmanuel Chasseigne\footnote{\texttt{<emmanuel.chasseigne@idpoisson.fr>}}}

\date{
Institut Denis Poisson (UMR CNRS 7013)\\
Universit\'e de Tours, Universit\'e d'Orl\'eans, CNRS\\
Parc de Grandmont\\
37200 Tours, France\\
}

\maketitle
\cleardoublepage

\centerline{\Large \bf Acknowledgements and various informations}

\ \\


\noindent{\large \bf AMS Class. No:}
35F21, 
49L20,   
49L25,   
35B51.  
\\

\noindent{\large\bf Acknowledgements: } The authors  were partially supported  by the ANR projects
HJnet (ANR-12-BS01-0008-01) and MFG  (ANR-16-CE40-0015-01).

\ \\
\noindent{\large\bf Informations: }
\begin{enumerate} 
    \item At the beginning of the book, just after the preface, a section called \emph{``Survival
        kit for the potential reader: how can this book be useful to YOU?''} aims at explaining how
        to enter into this book without reading it from the first pages. The answer depending of
        course on who you are and what you wish to find here.
    \item At the end of the book, in addition to the usual index, two appendices gather the main
        notations and assumptions which are used throughout this book. Also included is a list of
        the different notions of solutions as a quick reference guide.
\end{enumerate}

%
%

\cleardoublepage
\addcontentsline{toc}{chapter}{Preface}

\centerline{\large \bf Preface}
\vspace{2cm}

The genesis of this book can be traced back to the early 2010's. 

At that time, many researchers in the viscosity solutions community got interested in
Hamilton-Jacobi Equations set on networks. In order to avoid traffic jams on such research
themes, with Ariela Briani we decided to consider problems set in the usual euclidian space, but
having \emph{discontinuities}.

Of course we first considered the case of a codimension $1$ discontinuity. Meanwhile, we were
listening to talks on networks with interest, but as if they concerned different problems;
conversely, people working on networks were clearly thinking that we were addressing different
questions.

Then, inspired by the article of Bressan and Hong \cite{BH}, we moved to stratified problems, \ie
problems with discontinuities of any codimensions, but still in the whole euclidian space. We also
started thinking about possible generalizations to problems set in domains, bounded or not.

\

\noindent\emph{End of year 2017, starting the project ---} 
Three main facts convinced us that starting to write a book could be worth considering:
\begin{enumerate} 
    \item[$(i)$] Several discussions with Cyril Imbert made us realize that the methods used for
        networks could be useful for treating problems with codimension $1$ discontinuities; the
        article written in collaboration with Ariela Briani and Cyril Imbert \cite{BBCI} was a first
        step in this direction. But clearly more had to be said about this ``network approach''.
    \item[$(ii)$] The Tanker Problem exposed by Pierre-Louis Lions in one of his courses at the
        Coll\`ege de France was illuminating on the possible extensions of our stratified approach
        to treat a large variety of possibly singular boundary value problems without much
        additional effort.  
    \item[$(iii)$] Last---and perhaps least---, we noticed that some of the techniques we developed in
        the stratified context could be useful to extend the ``network approach'' to a
        multi-dimensional framework.  
\end{enumerate}

Though exploiting these ideas in publishing a series of articles was tempting, we decided instead to
start writing an ``evolutive book'': from the beginning, our plan was to get an online version
available to other researchers, that we would keep improving with possible contributions or help
from readers. And indeed, all the versions were modified by taking into account such
remarks as well as our own progress. 

This choice may appear quite particular as, in general, mathematical books are written when the theory
starts being well-established, key results have reached their (almost) definitive form and a global
understanding of the various phenomena has been validated by the community. 

But as we explained, we were not at all in such an idyllic situation when we started this project.
Our aim was to take time to produce a ``clean'' contribution to the subject, instead of polluting
literature with several unsatisfactory articles. By doing so, we decided to give ourselves time to correct
our own mistakes, be it minor ones in the proofs or errors in the strategy of those proofs, but also
in the presentation and articulation of the different results. 

The least we can say today is that we overused these possibilities.

\

\noindent\emph{Early 2018, writing the first pages ---}
The above paragraphs may give the impression that we were very ambitious but this was not entirely
the case. In terms of content, our initial plans for this book were rather modest: the main idea was
to gather in the same publication simplified versions of the comparison arguments for the hyperplane
case and the stratified framework which were known at that time. Concretely, this meant putting
together:
\begin{enumerate}
    \item[1.] our works on Ishii solutions for the hyperplane case \cite{BBC1,BBC2} showing the
        problems encountered by the classical viscosity solutions approach; 
    \item[2.] the comparison result for flux-limited solutions found in \cite{BBCI}, which was
        simplifying the Imbert-Monneau comparison arguments found in \cite{IM,IM-md}; 
    \item[3.] the Lions-Souganidis \cite{LiSo1,LiSo2} arguments for junction viscosity solutions;
    \item[4.] the stratified framework developed in \cite{AEYW}, with some ``easy'' extensions to
         state-constrained problems. 
\end{enumerate}
This project was thought of as a kind of compendium of 150-\emph{ish} pages about discontinuities
in Hamilton-Jacobi equations related to control problems, \ie restricting ourselves to the case of
convex Hamiltonians.  For the ``network approach'', our aim was both to clarify and simplify the
existing results and their proofs, as it seemed to us that there was some room to do so!
On the other hand, we wished to show that the ideas we had for stratified problems can be pushed
quite far, in particular with the aim of treating problems with boundary conditions---though we did
not realize how far and how concrete we could go at that time; but, in any case, we did not plan to
go too far in the treatment of these extensions.

However, even if we were not very ambitious with regards to generality, we were more so
on the contribution of this book: revisiting the recent progresses did not mean that we were merely
copy-pasting existing articles with few modifications. Instead, our goal was to highlight
the main common ideas, whether technical or more fundamental. With a better understanding of the
existing proofs, our hope was to simplify them as much as possible in order to promote further
developments.

All these original plans explain the organization of this book today: while thinking about all the
common points in several works, we decided from the beginning to dedicate an entire part,
Part~\ref{Part:prelim}, to the ``basic results'', which are common bricks, used very often under
perhaps slightly different forms, to prove the main results. This also has the advantage of
lightening the presentation of the main results and their proofs. But we cannot deny that this
creates a rather technical part that may also prove difficult to read, although it can be interesting to
see some classical ideas revisited in sometimes unusual ways.

Unfortunately (or fortunately?), even the first draft was not along the lines of our initial
objectives: we decided to add ``a little more'' material and the project soon reached almost 300 
pages---version~$1$, december 2018. The only rule we respected at that time was the framework of 
\emph{convex Hamiltonians for equations with a codimension~$1$ discontinuities}.

\

\noindent\emph{Year 2019, a reorganized and expanded second version ---}
We had to admit that our decision to restrict ourselves to convex Hamiltonians in
the case of codimension~$1$ discontinuities was a nonsense. Indeed, in the ``network approach'',
all the results inspired by the works of Imbert-Monneau \cite{IM,IM-md} were valid without much
change in their framework of \emph{quasi-convex Hamiltonians}. 

We then reorganized the book, building an entire part on this ``network approach''. Concerning the
arguments of Lions-Souganidis \cite{LiSo1,LiSo2} for junction viscosity solutions, we recall that
they work for Hamiltonians which are only continuous.

Moreover, we realized there was far more to be told than what we initially had in mind:
\begin{enumerate}
    \item[$(i)$] A comparison between the notions of \emph{Ishii, flux-limited and junction
        solutions} was not part of the initial plan, despite some results already appearing in
        \cite{IM,IM-md}. But, pushed by the challenging study of the convergence of the vanishing
        viscosity method---and the applications to KPP or Large Deviations type problems---,
        we discovered that we were able to make a quite complete and rather simple description of
        their links, in particular the conditions under which they are equivalent notions of
        solutions.
    \item[$(ii)$] We noticed that the stratified framework allowed us to deal with far more general
        situations than what we thought, including time-dependent stratifications, state-constrained
        and boundary value problems. Though all these themes were somehow present in the first
        version, we revisited all the results, simplified and sometimes generalized them. We even
        realized that some of what we considered as being the unavoidable ``basic tools'' had to be
        defined or used differently.
\end{enumerate}

\

\noindent\emph{The pandemic years, third version ---}
In 2020, the pandemic struck and kept us away from the project for more than a year for various
reasons. This imposed step back made us realize the numerous weaknesses of our first and second
versions. This led once more to a lot of additions and modifications in 2021--2022 which made the
project go far beyond the 500-page mark.

As the book unfolded, and even if this was not our objective at the beginning, we ended up
developing a very general framework to the cost of some complexities and technicalities. In
particular, it was challenging for the stratified approach to see how our initial ideas based on the
simple assumptions of \emph{normal controllability} and \emph{tangential continuity} could be pushed
to solve rather singular problems. And sometimes without much additional effort.

We are fully aware that the general framework we are presenting today is probably a bit complex when
considering simple and concrete applications. We hope it will not prevent or stop the reader from
delving into it. We have devoted a lot of time and effort to give non technical explanations as
much as possible.

We also made a point from the beginning, not only to give abstract results but also to explain how
they can be applied to concrete applications and contribute to new results. This explains the use of
``illustrative'' in the title: we have tried to incorporate as many examples and counter-examples as
we could, provide various applications and we have pointed out several puzzling open problems. 

As we also mentioned in various places, some situations can be treated with weaker assumptions,
through making good use of the specificities of each problem. But we are now convinced that the
assumptions we make are really needed in order to build quite a general framework, as
counter-examples show.

Though we did not fully implement them, we also tried to show how these approaches can be useful in
treating other situations like for instance non-local equations (trajectories with jumps) or
multi-dimensional networks.

\

\noindent\emph{Spring 2023, ending the project ---}
Five years after writing the first lines of this book, we decided to put an end to the writing
process of the project.

Version 4 reaching now more than 630 pages in its standard LateX version---a bit less in the
Springer Nature format---, we feel that it is now high time to publish what we somehow consider to
be a final version of the book. Since we have make even major changes right up to the end, we are
convinced that we could still improve the presentation. We could also probably add some other
results and implement new material. 

But of course, this would become an endless pursuit. 

Although in its form this book is far from what we initially had in mind, we have the
feeling that we approximately reached what was our aim: to present a collection of results,
approaches, situations that all share some common concepts and provide a framework which could make
everything coexist rather smoothly, even if everything is certainly still imperfect.

We hope that the reader of this manuscript will enjoy reading it and that its content will be
useful to anyone interested in these topics. Of course, we would be very happy to hear that some of
the open problems we mention here are finally solved in the future.\\

\

\parbox[c]{.3\hsize}{
G. Barles} \hfill\parbox[c]{.3\hsize}{E. Chasseigne}

\cleardoublepage
\addcontentsline{toc}{chapter}{Survival kit for the reader}

\centerline{\Large \bf Survival kit for the potential reader: }

\bigskip

\centerline{\large \emph{how can this book be useful to YOU?}}
\vspace{2cm}

Upon taking this book in your hands, looking at its size and content you might be a little
bit discouraged. Furthermore, the idea that you have to read and digest the huge first part
called the ``Toolbox''---containing the basic results which are useful to solve problems involving
Hamilton-Jacobi-Bellman Equations and/or deterministic control problems with discontinuities---can
be more than frightening.

We admit that this part is unavoidably ``a bit technical'', hard to read without some serious
motivation... Which we hope can be found in the rest of the book!  But, and this may be good news,
we think most of our readers will skip the ``Toolbox'', at least parts of it.
We have however to issue a warning:

\begin{center}
\begin{minipage}{0.8\textwidth}
\emph{This book is not designed for complete beginners in the theory of viscosity solutions nor in
deterministic optimal control problems.} 
\end{minipage}
\end{center}

Indeed, it seems clear to us that addressing problems on Hamilton-Jacobi Equations and/or
deterministic control problems with discontinuities requires reasonable mastery of such problems
in the continuous case. 

More precisely, we find it unavoidable to assume that the reader is at least familiar with some
notions, results and their related proofs such as: comparison results for viscosity solutions;
stability results for viscosity solutions; connections between standard finite horizon control
problems with Hamilton-Jacobi-Bellman Equations using the viscosity solution approach. A good test
in this direction consists of checking that you are not lost while browsing
Chapter~\ref{chap:BasicFram}.\\

Coming back to the toolbox, we have tried to draft all the proofs in the book by emphasizing the
role of the related key bricks (introduced in this toolbox), and we did it in a manner that the
arguments remain readable without knowing the details of such bricks. In this way, one can avoid
reading the different independent sections of Part~\ref{Part:prelim} at first, before being
completely convinced that it may be necessary.

On the other hand, depending on who you are and what you hope to find in this book, you may consider
different (and safer!) entry points than the ``Toolbox''. Here are some suggestions for different
readers:

\begin{enumerate}
    \item[$(i)$] {\bf You are an ``enlightened beginner''} and want to learn some basics about HJB-Equations with
        discontinuities: Part~\ref{part:codim1} is certainly the most unavoidable. Starting from
        Chapter~\ref{chap:BasicFram} which exposes the standard continuous case, this part then
        goes on by describing all the challenges and potential solutions at hand in the rather simple context of a
        codimension $1$ discontinuities. Yet the difficulty of this part is to extract a clear global
        vision and we try to provide our point of view in Chapter~\ref{sec:SCQ}.
    \item[$(ii)$] {\bf You are interested in stratified problems:} this clearly requires a
        non-negligible investment since it seems difficult to avoid first reading
        Chapter~\ref{chap:control.tools} on \emph{Control Tools}, even just to get the notations. Then
        you can start reading Part~\ref{stratRN}: we have tried to point out the main ideas to keep
        in mind, starting from the easiest case before going towards the most sophisticated
        ones. We hope that the general treatment of singular boundary conditions in non-smooth
        domains, Part~\ref{S-BC}, will be a sufficient motivation for enduring all the
        difficulties! The applications contained in Chapter~\ref{chap:appl-strat} may also motivate you.
    \item[$(iii)$] {\bf You are interested in HJ-Equations on networks:} Part~\ref{part:NA} is made
        for you! Of course, we do not really treat networks (we only consider two-branch
        junctions) but this part contains ideas---strongly inspired from Imbert-Monneau and
        Lions-Souganidis---which we have simplified as much as we could, that you will certainly be
        able to use in far more complicated situations. You can also have a look at
        Chapter~\ref{chap:networks} for some ideas on multi-dimensional networks.
    \item[$(iv)$] {\bf You are interested in scalar conservation laws and the connections with
        HJ-Equations:} it is brave of you to be here! As a reward for such audacity, we have
        written Chapter~\ref{Embl-Exemple} especially for you! We hope to have done a good enough
        job there.
\end{enumerate}

%
%

\setcounter{page}{0}

\tableofcontents


\chapter*{Introduction}
\addcontentsline{toc}{part}{Introduction}
\fancyhead[LE,RO]{\thepage}
\fancyhead[CE]{Barles \& Chasseigne}
\fancyhead[CO]{HJ-Equations with Discontinuities: Introduction}


\section*{Viscosity solutions and discontinuities}
\addcontentsline{toc}{section}{Viscosity solutions and discontinuities}
\label{chap:Intro}

In 1983, the introduction of the notion of viscosity solutions by Crandall and
Lions \cite{CL} solved the main questions concerning first-order
Hamilton-Jacobi Equations (HJE in short), at least those set in the whole space
$\R^N$, for both stationary and evolution equations:  this framework provided
the right notion of solutions for which uniqueness and stability hold, allowing
to prove (for example) the convergence of the vanishing viscosity method. In
this founding article the definition was very inspired by the works of Kru\v
zkov \cite{SNK1,SNK2,SNK3,SNK4} and, in fact, viscosity solutions appeared as
the $L^\infty$-analogue of the $L^1$-entropy solutions for scalar conservation
laws.

This initial, rather complicated Kru\v zkov-type definition, was quickly
replaced by the present definition, given in the article of Crandall, Evans and
Lions \cite{CEL}, emphasizing the key role of the Maximum Principle and of the
degenerate ellipticity, thus preparing the future extension to second-order
equations.

\

\noindent\textsc{A simple, universal and efficient notion of solution}

The immediate success of the notion of viscosity solutions came from both its
simplicity but also universality: only one definition for all equations, no
matter whether the Hamiltonian was convex or not. A single theory was providing
a very good framework to treat all the difficulties connected to the
well-posedness (existence, uniqueness, stability...) but it was also
fitting perfectly with the applications to deterministic control problems,
differential games, front propagations, image analysis etc.

Of course, a second key breakthrough was made with the first proofs of
comparison results for second-order elliptic and parabolic, possibly
degenerate, fully nonlinear partial differential equations (pde in short) by
Jensen \cite{J2nd} and Ishii \cite{I2nd}. They allow the extension of the
notion of viscosity solutions to its natural framework and open the way to more
applications. This extension
definitively clarifies the connections between viscosity solutions and the
Maximum Principle since, for second-order equations, the Maximum Principle is a
standard tool and viscosity solutions for degenerate equations are those for
which the Maximum Principle holds when testing with smooth test-functions. 

The article of Ishii and Lions \cite{IL} was the first one in
which the comparison result for second-order equations was presented in the
definitive form; we recommend this article which contains a lot of results and
ideas, in particular in using the ellipticity in order to obtain more general
comparison results or Lipschitz regularity of solutions.

We refer to the User's guide of Crandall, Ishii and Lions \cite{Users} for a
rather complete introduction of the theory. See also Bardi and
Capuzzo-Dolcetta\cite{BCD} and Barles \cite{Ba} for first-order equations,
Fleming and Soner \cite{FS} for second-order equations together with
applications to deterministic and stochastic control, Bardi, Crandall, Evans,
Soner and Souganidis \cite{BCESS} ot the CIME course \cite{CIME} for a more
modern presentation of the theory with new applications. 

\

\noindent\textsc{Discontinuities, a potential weakness of viscosity solutions}

Despite all these positive points, the notion of viscosity solutions had a
little weakness: it only applies with the maximal efficiency when solutions are
continuous and, this is even more important, when the Hamiltonians in the
equations are continuous. This fact is a consequence of the keystone of the
theory, namely the comparison result, which is mainly proved by the  ``doubling
of variables'' technique, relying more or less on the continuity of both the
solutions and the Hamiltonians.

Yet, a definition of discontinuous solutions has appeared very early (in 1985)
in Ishii \cite{I1} and a first attempt to use it in applications to control
problems was proposed in Barles and Perthame \cite{BP1}. The main contribution
of \cite{BP1} is the ``half-relaxed limits method'', a stability result for
which only a $L^\infty$-bound on the solutions is needed. But this method,
based on Ishii's notion of discontinuous viscosity solutions for
discontinuous Hamiltonians, uses discontinuous solutions more as an
intermediate tool than as an interesting object by itself.

\

\noindent\textsc{The end of universality?}

However, in the late 80's, two other types of works considered discontinuous
solutions and Hamiltonians, breaking the universality feature of viscosity
solutions. The first one was the study of measurable dependence in time in
time-dependent equation (\cf Barron and Jensen \cite{BJ-meas}, Lions and
Perthame \cite {LP-meas}, see also the case of second-order equations in
Nunziante \cite{Nun1,Nun2}, Bourgoing \cite {Bou1,Bou2} with Neumann boundary
conditions, and Camilli and Siconolfi \cite{CaSi}): in these works, the
pointwise definition of viscosity solutions has to be modified to take into
account the measurable dependence in time. It is worth pointing out that there
was still no difference between convex and non-convex Hamiltonians.

On the contrary, Barron and Jensen \cite{BJ} in 1990 considered semi-continuous
solutions of control problems (See also \cite{Ba-GV} for a slightly simpler
presentation of the ideas of \cite{BJ} and Frankowska \cite{Fran}, Frankowska
and Plaskacz \cite{Fran-P}, Frankowska and Mazzola \cite{Fran-M} for different
approaches): they introduced a particular notion of viscosity solution which
differs according to whether the control problem consists in minimizing some
cost or maximizing some profit; thus treating differently convex and concave
Hamiltonians. This new definition had the important advantage to provide a
uniqueness result for lower semi-continuous solutions in the case of  convex
Hamiltonians, a very natural result when thinking in terms of optimal control.

In the period 1990-2010, several attempts were made to go further in the
understanding of Hamilton-Jacobi Equations with discontinuities. A pioneering
work is the one of Dupuis \cite{Du} whose aim was to construct and study a
numerical method for a calculus of variation problem with discontinuous
integrand, motivated by a Large Deviations problem.  Then, control problems
with a discontinuous running cost were addressed by Garavello and Soravia
\cite{GS1,GS2} and Soravia \cite{So}  who highlight some non-uniqueness feature
for the Bellman Equations in optimal control, but identify the maximal and
minimal solutions. To the best  of our knowledge, all the uniqueness results
use either a special structure of the discontinuities or different notions
solutions, which are introduced to try to tackle the main difficulties as in
\cite{DeZS,DE,GGR,GH, H1} or an hyperbolic  approach as in \cite{AMV,CR}. For
the boundary conditions, Blanc \cite{Bl1,Bl2} extended the approaches found in \cite{BP1} and
\cite{BJ} to treat problems with discontinuities in the boundary
data for Dirichlet problems. Finally, even the case of measurability in the
state variable was considered for Eikonal type equations by Camilli and
Siconolfi~\cite{CaSi-meas}.

Before going further, we point out that we do not mention here
$L^p$-viscosity solutions nor viscosity solutions for stochastic pdes, two very
interesting subjects but too far from the scope of this book. 

\

\noindent\textsc{Towards more general discontinuities}

In this period, the most general contribution for first-order Hamilton-Jacobi-Bellman Equations was
the work of Bressan and Hong \cite{BH} who considered the case of control problems in {\em
stratified domains}. In their framework, the Hamiltonians can have discontinuities on submanifolds
of $\R^N$ of any codimensions which form a Whitney stratification and the viscosity solutions
inequalities are disymmetric between sub and supersolutions (we come back on this important point
later on). In this rather general setting, they are able to provide comparison results by combining
pde and control methods. Of course, we are very far from the context of a universal definition but
it seems difficult to have more general discontinuities. Before going further, we refer the reader
to Whitney \cite{W1,W2} for the notion of {\em Whitney stratified space}. 

\

\noindent\textsc{Networks}

In the years 2010's, a lot of efforts have been spent to understand Hamilton-Jacobi Equations on
networks and, maybe surprisingly, this had a key impact on the study of codimension~$1$
discontinuities in these equations.  An easy way to understand why is to look at an HJ Equation set
on the real line $\R$, with only one discontinuity at $x=0$. Following this introduction, it seems
natural to jump on  to Ishii's definition and to address the problem as an equation set on $\R$. But
another point of view consists in seeing $\R$ as a network with two branches $\R^-$ and $\R^+$. This
way, $x=0$ becomes the intersection of the two branches and it is conceivable that the
test-functions could be quite different in each branch, leading to a different notion of solution.
Moreover, a ``junction condition'' is needed at $0$ which might come from the two Hamiltonians
involved (one for each branch) but also a specific inequality at $0$ coming from the model and the
transmission condition we have in mind.  Therefore, at first glance, these ``classical approach''
and  ``network approach''  seem rather different.

Surprisingly (with today's point of view), these two approaches were investigated by different
people and (almost) completely independently until Briani, Imbert and the authors of this book made
the simple remark which is described in the last above paragraph. But, in some sense, this ``mutual
ignorance'' was a good point since different complementary questions were investigated and we are
going to described these questions now.

For the ``classical approach'', in the case of the simplest codimension $1$
discontinuity in $\R$ or $\R^N$ and for deterministic control problems, \ie
with convex Hamiltonians, these questions were
\begin{enumerate}
\item[$(i)$] Is Ishii's definition of viscosity solutions providing a unique
    solution which is the value function of an associated control problem?
\item[$(ii)$] If not, can we identify the minimal and maximal solutions in
    terms of value functions of \adhoc control problems?
\item[$(iii)$] In non-uniqueness cases, is it possible to recover uniqueness by
    imposing some additional condition on the discontinuity?
\item[$(iv)$] Can the limit of the vanishing viscosity method be identified? Is
    it the maximal or minimal solution? Or can it change depending on the
        problem?
\end{enumerate}
These questions were investigated by Rao \cite{rao:pastel-00927358, Rao-2013},
Rao and Zidani \cite{Rao-Zid-2012}, Rao, Siconolfi and Zidani \cite{RSZ} by
optimal control method, and  Barles, Briani and Chasseigne \cite{BBC1,BBC2} by
more pde methods. In \cite{BBC1,BBC2}, there are some complete answers to
questions $(i)$ and $(ii)$, almost complete for $(iii)$ and really incomplete
for $(iv)$.

For the ``network approach'',  in the case of two (or several) $1-$dimensional
(or multi-dimensional) branches, the questions were different and the convexity
of the Hamiltonians appears as being less crucial:
\begin{enumerate} 
\item[$(v)$] What is the correct definition of solution at the junction? What
    are the different possible junction conditions and their meanings in the
        applications?
\item[$(vi)$] Does a comparison result for such network problems hold?
\item[$(vii)$] Does the Kirchhoff condition (involving derivatives of the
    solution in all branches) differ from tangential conditions (which just
    involve tangential derivatives)?  
\item[$(viii)$] What are the suitable assumptions on the Hamiltonians to get
comparison?  
\item[$(ix)$] Can we identify the limit of the vanishing viscosity
    method?
\end{enumerate}
Questions $(v)$-$(vi)$ were investigated under different assumptions in
Schieborn \cite{Sc}, Camilli and Marchi \cite{CM}, Achdou, Camilli,
Cutr{\`{\i}} and Tchou \cite{ACCT}, Schieborn and Camilli \cite{ScCa}, Imbert,
Monneau and Zidani \cite{IMZ}, Imbert and Monneau \cite{IM} for $1$-dimensional
branches and Achdou, Oudet and Tchou \cite{AOT,AOTbis}, Imbert and Monneau
\cite{IM-md} for all dimensions; while Graber, Hermosilla and Zidani \cite{GHZ}
consider the case of discontinuous solutions. The most general comparison
result  (with some restrictions anyway) is the one of Lions and Souganidis
\cite{LiSo1,LiSo2} which is valid with very few, natural assumptions on the
Hamiltonians, and not only in the case of Kirchhoff conditions but also for
general junction conditions. It allows to answer in full generality to question
$(ix)$ which is also investigated in Camilli, Marchi and Schieborn
\cite{CMS-VV}.

In fact, Lions and Souganidis use a notion of solution which we call in this book ``junction
viscosity solution'', rather close to the classical notion of viscosity solutions; the
only difference which is imposed by the network framework is the space of test-functions but this is
a common feature for all the notions of solution in this context. Because of this similarity, the
half-relaxed limits' method extends without any difficulty and, taking into account the very general
ideas of their comparison result, almost all the above questions seem to be solved by this notion of
solution.

Two questions still remain however: on one hand, despite of its generality, the comparison result of Lions
and Souganidis requires in higher dimensions some unnatural hypotheses; on the other hand, this result
is originally proved in \cite{LiSo1,LiSo2} for Kirchhoff type junction conditions which is not the most
natural conditions for control problems, but which appear when studying the convergence of the vanishing
viscosity method. Hence, a very concrete question is the following: in the case of convex or concave
Hamiltonian, is it possible s to give formulas of representation for such problems with Kirchhoff type junction
conditions?  To answer this question, it seems clear that one has to investigate the connections
between Kirchhoff type junction conditions and ``flux-limited conditions'' in the terminology of
Imbert and Monneau \cite{IM,IM-md} which are the natural junction conditions for control problems.

The extensive study of ``flux-limited conditions'' by Imbert and Monneau \cite{IM,IM-md} uses the
notion of ``flux-limited solutions'': contrarily to the notion of ``junction viscosity solution'',
this notion is less general and requires quasi-convex Hamiltonians on each branch of the network. It
has also the defect to lead to a rather complicated (and limited) stability result. But it perfectly
fits with control problems and the comparison result is proved under very natural and general
assumptions.

In this book, we completely describe these two notions of solutions and theirs properties but we
also show the connections between general Kirchhoff conditions and flux-limited conditions in the
quasi-convex case, allowing the complete identification of the vanishing viscosity limit.

\section*{Key considerations related to discontinuities}
\addcontentsline{toc}{section}{Key considerations related to discontinuities}

In this short section, our aim is to highlight a few simple and fundamental ideas that pervade the
whole book.

Let us begin with saying that in order to understand Hamilton-Jacobi Equation with discontinuities,
a first natural step is to look at deterministic control problems. Since our aim is to extend
viscosity solutions theory to this discontinuous framework---in particular the pillars of the theory
which are the comparison and stability results---we can only do so under some assumptions which
ensure that the value function is continuous and the unique solution of the associated
Hamilton-Jacobi-Bellman Equation. Indeed, these properties are standard consequences of the
comparison result for this equation.

While looking at problems with codimension~$1$ discontinuities, one quickly realizes that the
standard definition of viscosity solutions in the sense of Ishii, in particular the subsolution
condition, is not strong enough to imply uniqueness; in the worst cases, the subsolution condition
completely ignores the possibilities that the control problem offers on the discontinuity. This
is particularly the case when the only optimal trajectory for the controller consists in staying on
the discontinuity, because the situation is far more favorable there. The reader may have in mind the
example of a car ride where taking advantage of $1$-dimensional highways allows to reach the
destination must faster; if the subsolution condition does not see the highway, we clearly get
meaningless subsolutions.

This is the first point to keep in mind for Hamilton-Jacobi-Bellman Equation, \ie for
Hamilton-Jacobi Equation with convex Hamiltonians: 

\keypoint{Key point 1}{A subsolution condition is missing on the discontinuities and we have to
super-impose a right one on each of them in order to build a satisfactory theory.}

On the other hand, the example of the car ride and the highway suggests a second key remark: if you
can enter the highway everywhere, you can expect that your travel time does not depend too much on
your departure point, in the sense that, if you start from two close points, the two travel times
are almost the same. But if the highway has only few entrances and if you take two close points, one
on the highway, one outside, both being far from an entrance, the travel times can be very
different. Hence such situations generate a discontinuity for the value function (that is, the
travel time to a fixed destination) and we have to rule them out.

\begin{figure}[htp]
   \begin{center}
   \includegraphics[width=0.6\textwidth]{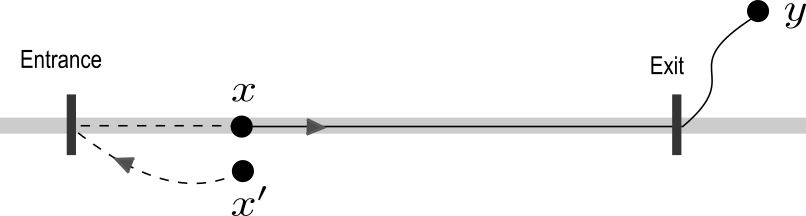}
   \caption{Highway generating discontinuities}
   \label{fig:highway} 
   \end{center}
\end{figure}

In the above example, being located on the discontinuity is favorable for the controller but you may also
imagine the opposite situation: if the highway is replaced by a very muddy road where the velocity
is far slower than every else. Then you may probably want to be able to get out of this road.
This is the second important point to keep in mind: 

\keypoint{Key point 2}{In such problems having discontinuities, the normal controllability---or normal
    coercivity---of the problem is fundamental.}

This property appears below under either the form \NCe, \ie  Normal Coercivity for the equation or
\NCBCL, \ie Normal Controllability for the control problem. But most of the time they are exactly the
same.
 
As the above examples shows, in the framework of control problems, this property means that one
should be able either to quickly reach the discontinuity (we will use it for discontinuities of any
codimension) or, on the contrary, to leave it in any directions, in order to take advantage of a
more favorable situation in terms of cost. Such assumption also ensures that this potentially
favorable situation is ``seen'' by subsolutions provided that the right conditions are imposed on
the discontinuities. Finally, at the equation level, this is translated into a partial
coercivity-type assumption in the normal coordinates of the gradient on the associated Hamiltonian.
        
The last key idea is the tangential continuity, denoted below by either \TC for the
equation or \TCBCL for the control problem. Roughly speaking, if we face a discontinuity
$\mathcal{D}$ which is an affine subspace, the Hamiltonian has to satisfy locally similar continuity
hypotheses as those used in standard comparison in $\R^N$ (or $\R^N \times (0,\Tf)$) on each affine
subspace which is parallel to $\mathcal{D}$, with respect to the coordinates of this subspace.

\keypoint{Key point 3}{Some natural continuity assumptions should hold with respect to each strata.}

We do insist on the fundamental role played by assumptions \NCe-\NCBCL and \TC-\TCBCL
throughout all the parts of this book. Not only are they key ingredients for the comparison result
between sub and supersolutions, but also for the stability and even the connections with control
problems, \ie to actually prove that the value function is sub and supersolution with the adapted
definition.

%
%
%
\section*{Overview of the content}
\addcontentsline{toc}{section}{Overview of the content}

As the reader has probably understood, this book aims at considering various Hamilton-Jacobi and
control problems with different types of discontinuities. Our intention is to describe the different
approaches to treat them and build a consistent framework in which they can fit altogether. Let us
now sketch the content of this book part by part.

\

\noindent\textsc{Overview of Part~\ref{Part:prelim} ---}
Thinking about all the common points that can be found in the works mentioned
in the historical introduction above, and because of the central roles played by \NCe-\NCBCL and
\TC-\TCBCL, we have decided to dedicate an entire part to the ``basic results'', which are common
bricks, used very often under perhaps slightly different forms. 

This organisation has the advantage to lighten the presentation of the main results
and their proofs, but this clearly creates a rather technical---and perhaps difficult
to read---first part. We think anyway that collecting some classical ideas, sometimes 
revisited in unusual ways, presents sufficient advantages to accept this flaw.

\

\noindent\textsc{Overview of Parts~\ref{part:codim1} \&~\ref{part:NA} ---}
The first problems we address concern ``simple'' codimension $1$ discontinuities, \ie a discontinuity
along an hyperplane or an hypersurface in the whole space $\R^N$.  For these problems, we provide in
Part~\ref{part:codim1} a full description of the ``classical approach''. By this, we mean the
results that can be obtained by using only the standard notion of viscosity solutions. In
Part~\ref{part:NA}, we describe the ``network approach'', including different comparison proofs (the
Lions-Souganidis one and the Barles-Briani-Chasseigne-Imbert one) and stability results. We also
analyze their advantages and disadvantages, and the connections between all the notions of
solutions.

The main results of these parts are the following.
\begin{enumerate}
    \item[$(i)$] Identification of the minimal viscosity supersolution and maximal viscosity
        subsolution with explicit controls formulas.  Furthermore, we provide an easy-to-check
        condition on the Hamiltonians ensuring that these minimal supersolution and maximal
        subsolution are equal, \ie that there is a unique viscosity solution. This condition turns
        out to be useful in different applications.  
    \item[$(ii)$] For the different notions of solutions in the ``network approach'', we provide
        comparison and stability results, and a complete analysis of the connections between these
        different types of solutions (classical Ishii viscosity solutions, flux-limited solutions
        and junction viscosity solutions).
    \item[$(iii)$] Several versions of the convergence of the vanishing viscosity method, for convex
        and non-convex Hamiltonians, each of them relying on a particular notion of solution; the
        most complete form uses all the results of $(i)$ and $(ii)$ above, in particular the links
        between the different notions of solutions in the ``network approach''.
\end{enumerate}
The reader who wants to have a quick idea of all these results can take a look
at Section~\ref{Embl-Exemple}. This section gives a flavor of them in a simple
framework, the Hamilton-Jacobi analogue of $1$-d scalar conservation laws with
a discontinuous flux.

\

\noindent\textsc{Overview of Part~\ref{stratRN} ---} 
This part is devoted to the case of time-dependent ``stratified problems'' in the
whole cylinder $\R^N\times[0,\Tf)$, \ie  the case where discontinuities of any
codimensions can appear, provided they form a Whitney stratification. In this part, we describe the
extensions of the works by Bressan and Hong \cite{BH} and by the authors in \cite{AEYW}, with a lot
of further applications.

We point out anyway two main differences with \cite{AEYW}: first, we introduce a notion of {\em weak
stratified subsolutions}\footnote{The situation for supersolutions is simpler since stratified
supersolutions are just Ishii supersolutions.} where, on each manifold of the Whitney
stratification, we only impose inequalities associated to dynamics which are tangent to the
manifold. Such subsolutions are not assumed to satisfy the usual ``global'' Ishii subsolution inequality on the
manifolds of codimension bigger than $1$; hence they are not a priori Ishii subsolutions. On the
contrary, {\em strong stratified subsolutions}---as used in \cite{AEYW}---are weak stratified
subsolutions, which are also Ishii subsolutions.

In the stratified setting, the notion of subsolution that has to be imposed on the discontinuities
is a key issue: the concepts of weak and strong stratified subsolutions turn out to be the weakest
and the strongest possible versions. In the different works on the subject, various other type of
definitions appears, from a ``quasi-strong'' notion in Bressan and Hong \cite{BH} to the use of
``essential dynamics'' in Rao \cite{rao:pastel-00927358, Rao-2013}, Rao and Zidani
\cite{Rao-Zid-2012}, Rao, Siconolfi and Zidani \cite{RSZ}  and Jerhaoui and Zidani~\cite{JZ} where
the authors try to obtain the best possible inequalities from the control point of view. 

Despite being rather natural from the control point of view, the notion of weak stratified
subsolutions has the defect to allow ``artificial values'' on the discontinuities of the equation
since no connection between these values on the different parts of the Whitney stratification is
imposed by the weak subsolution inequalities. This is the second key difference with \cite{AEYW} where
the ``global'' Ishii subsolution inequality and \NCe (or \NCBCL) imply the ``regularity of
subsolutions'', \ie the fact that on a discontinuity, the values of a subsolution is the $\limsup$
of its values outside this discontinuity. Hence strong stratified subsolutions are necessarily
``regular'' while it may not be the case for the weak ones. And concerning the definitions with
``essential dynamics'',  we point out that, in general, the subsolution conditions which are imposed
imply the regularity of the subsolutions and conversely the inequalities associated to with
``essential dynamics'' are automatically satisfied by regular subsolutions.

As it is already remarked in \cite{AEYW}, the regularity property for subsolutions is playing a very
important role for all the results, and more particularly for the comparison one. 

To summarize the content of Part~\ref{stratRN}, let us first mention that all the results of
\cite{AEYW} hold for \emph{regular weak stratified subsolutions}. But more precisely:
\begin{enumerate}
    \item[$(i)$] Regular weak stratified subsolutions are \emph{strong stratified
        subsolutions} under suitable assumptions, which are, in our opinion, the natural
        hypotheses to be used in this framework. 

    \item[$(ii$)] The comparison result between either regular weak or strong stratified
        subsolutions and supersolutions; it uses in a key way \NCBCL-\TCBCL but also standard
        reductions presented in Part~\ref{Part:prelim}. 

    \item[$(iii)$] We present different stability results where we improve the one given in
        \cite{AEYW} by taking into account changes in the structure of discontinuities: indeed we
        handle cases where some discontinuities may either disappear or appear when we pass to the
        limit. It is worth remarking that the notion of weak stratified subsolutions has the
        advantage to simplify the proofs of these stability results, even if the regularity of the
        limiting subsolution becomes a problem.

    \item[$(iv)$] We provide conditions under which classical viscosity (sub)solutions are
        stratified (sub)solutions. Under these conditions, classical viscosity solutions and
        stratified solutions are the same, which allows to treat in a rather precise way some
        applications. This applies in particular to KPP-type problems, even in rather complicated
        domains. Indeed, we can take advantage at the same time of the good properties of viscosity
        solutions in terms of stability, and the uniqueness of stratified solutions.
\end{enumerate}

Different approaches for control problems in stratified frameworks, more in the
spirit of Bressan and Hong \cite{BH} have been developed by Hermosilla, Wolenski and
Zidani \cite{Her-Wol-Zid-2017} for Mayer and Minimum Time problems, Hermosilla
and Zidani \cite{Her-Zid-2015} for classical state-constrained problems,
Hermosilla, Vinter and Zidani \cite{Her-Vin-Zid-2017} for (very general)
state-constrained problems, including a network part.

\

\noindent\textsc{Overview of Part~\ref{S-BC} ---} Here we extend these ideas to consider ``stratified
problems'' set in a ``stratified domain'' with state-constraints boundary conditions. Without
enterinf into too much details here, the reader may imagine that a ``stratified domain'' may be far
from being smooth and corners are not the only source of irregularity for the boundary. Indeed, the
discontinuities in the data itself have to be taken into account.

Concretely, the advantage of the stratified formulation is to provide an approach where:
\begin{enumerate}
    \item[$(i$)] One can treat various boundary conditions (Dirichlet, Neumann, sliding boundary
        conditions,...) in the same framework. 
    \item[$(ii)$] The boundary of the domain need not be smooth, nor does the data.
    \item[$(iii)$] The mixing of mixing boundary conditions in some rather exotic way is allowed.
\end{enumerate}

A typical example of mixing singular boundary conditions is the Tanker problem, presented at the
beginning of this part.

Roughly speaking, all the results of Part~\ref{stratRN} can be extended to this more general
framework since, essentially, the boundary and the discontinuities in the boundary conditions just
create new parts of the stratification and new associated Hamiltonians. Only the ``one-sided
feature'' coming from the absence of exterior controllability at the boundary generates some
technical difficulties. For instance, the regularity of subsolutions which comes automatically from
\NCBCL in $\R^N$ is not so simple here. We show in this part how to reformulate classical boundary
conditions and conclude with the non-standard example of the Tanker problem.

\

\noindent\textsc{Overview of Part~\ref{part:compl-appl} ---} In this last part we collect some concrete applications where the
stratified approach helps or may help solving some problems. The study of fronts propagations for
KPP Equation via the Freidlin's approach (\cite{F}) is a classical playground for viscosity solutions and
we investigate the type of new results that the methods of this book allow to prove. We also propose some
ideas to address problem with jumps or set on ``stratified networks''.

\

\noindent\textsc{Appendices ---} As this book contains quite a lot of notions, definitions of solutions and
properties, we decided to reference all of them in those two appendices (Notations, and Assumptions).


\part{A Toolbox for Discontinuous Hamilton-Jacobi Equations and Control Problems}
\label{Part:prelim}
\fancyhead[CO]{HJ-Equations with Discontinuities: A Toolbox for Discontinuous Problems}


\chapter{The Basic Continuous Framework Revisited}
\label{chap:BasicFram}

\abstract{In this first chapter, the most classical results in the continuous framework are presented. The
assumptions and methods are discussed and revisited in order to introduce and partially justify the
general approach that is developed afterwards.}

Viscosity solutions' theory relies on two types of key results: comparison results and stability
results. If the ``half-relaxed limits'' method provides stability in a very general discontinuous
framework where both solutions and Hamiltonians may be discontinuous (see Section~\ref{sect:stab}),
the situation is completely different for comparison. If most of the classical arguments for
comparison can handle discontinuous sub and supersolutions, none of them can really handle
discontinuous Hamiltonians, even in the simplest cases of discontinuities.

As indicated in the abstract, we first describe one of the most classical result in the
continuous framework and in the simplest framework; it explains the connections between deterministic
optimal control problems and Hamilton-Jacobi-Bellman Equations, with the role played by viscosity solutions.
Even if our presentation is certainly too sketchy, the reader will notice that this result relies on two key arguments which,
throughout this book, will also be at the origin of most of the presented results: the Dynamic Programming Principle and the
comparison result.

In this chapter, we assume that the reader is more or less familiar with such approach and classical results.
And we refer to well-known references on this subject for more details: Lions \cite{L}, Bardi and
Capuzzo-Dolcetta \cite{BCD}, Fleming and Soner \cite{FS}, the CIME courses \cite{BCESS,CIME} and
Barles \cite{Ba}.

\section{The value function and the associated pde}

We consider a finite horizon control problem\index{Control problem!basic} in $\R^N$ on the time
interval $[0,\Tf]$ for some $\Tf>0$, where, for $x\in \R^N$ and $t\in [0,\Tf]$, the state of the system
is described by the solution $X(\cdot)$ of the ordinary differential equation
$$
    \dot X(s)= b(X(s),t-s,\alpha(s))\; ,\; X(0)=x \in \R^N \;.
$$
Here, $\alpha(\cdot) \in \mathcal{A}:=L^\infty(0,\Tf;A)$ is the control which takes values in the
compact metric space $A$ and $b$ is a continuous function of all its variables. More precise
assumptions are introduced later on.

For a finite horizon problem, the value function is classically defined by
\begin{align*}
    U(x,t) = \inf_{\alpha(\cdot) \in \mathcal{A}}& \biggl\{\int_0^{t} 
    l(X(s),t-s,\alpha(s))\exp\left(-\int_0^s c(X(\tau),t-\tau,\alpha(\tau)) d\tau \right) \ds\\ 
    & + \u0 (X(t))\exp\left(-\int_0^t c(X(\tau),t-\tau,\alpha(\tau)) d\tau \right) \biggr\}\; ,
\end{align*}
where $l$ is the running cost, $c$ is the discount factor and $\u0$ is the final cost. All these
functions are assumed to be continuous on $\R^N \times [0,\Tf] \times A$ (for $l$ and $c$) and on
$\R^N$ (for $\u0$) respectively.

The most classical framework use the following assumptions which will be refered below as

\label{page:HBACPc} 
\begin{assumption}{\HBACPc}{Basic Assumptions on the Control Problem -- Classical case.}\\[-1.2cm]
\begin{enumerate} 
    \item[$(i)$] The function $\u0: \R^N\to \R$ is a bounded, uniformly continuous
            function.
    \item[$(ii)$] The functions $b,c,l$ are bounded, uniformly continuous on $\R^N \times [0,\Tf]
        \times A$.
    \item[$(iii)$] There exists a constant $C_1>0$ such that, for any $x,y \in \R^N$, $t \in [0,\Tf]$,
        $\alpha \in A$, we have $$ |b(x,t, \alpha)-b(y,t,\alpha)| \leq C_1 |x-y|\; .$$
\end{enumerate} 
\end{assumption}

One of the most classical results connecting the value function with the associated
Hamilton-Jacobi-Bellman Equation is the \index{Viscosity solutions!value functions as}
\begin{theorem}\label{MBR:C-HJ}
    If Assumption \HBACPc holds, the value function  $U$ is continuous on $\R^N \times [0,\Tf]$ and is
    the unique viscosity solution of 
    \begin{align}      
        u_t + H(x,t,u,D_x u) = 0 \quad \hbox{in  }\R^N \times (0,\Tf)\;,\label{be-control} \\
         u(x,0) = \u0(x) \quad \hbox{in  }\R^N \;.\label{bid-control}
    \end{align}
    where 
    $$
        H(x,t,r,p):= \sup_{\alpha \in A}\,\left\{-b(x,t, \alpha)\cdot p + 
        c(x,t, \alpha)r-l(x,t,\alpha)\right\}\; .
    $$ 
\end{theorem}

In Theorem~\ref{MBR:C-HJ}, we have used the notation $u_t$ for the time derivative of the function
$(x,t)\mapsto u(x,t)$ and $D_x u$ for its derivatives with respect to the space variable~$x$. These
notations will be used throughout this book.

\begin{proof}[Sketch of Proof] 
Of course, there exists a lot of variants of this result with different assumptions 
on $b,c,l$ and $\u0$ but, with technical variants, the proofs use mainly the same steps.

\noindent \textbf{(a)} The first one consists in proving that $U$ is continuous and satisfies a Dynamic
    Programming Principle (DPP in short)\index{Dynamic Programming Principle!simplest form}, $i.e.$ that for any
    $0<h<t$,
\begin{align*}
     U(x,t) = \inf_{\alpha(\cdot) \in \mathcal{A}}& \bigg\{\int_0^{h} 
     l(X(s),t-s,\alpha(s))\exp\left(-\int_0^s 
     c(X(\tau),t-\tau,\alpha(\tau)) d\tau \right) \ds  \\
    &+ U (X(h),t-h)\exp \left(-\int_0^h c(X(\tau),t-\tau,\alpha(\tau)) 
    d\tau \right) \bigg\}\; .
\end{align*}
This is obtained by using the very definition of $U$ and taking suitable controls.

    \noindent\textbf{(b)} If $U$ is smooth, using the DPP on $[0,h]$ and performing expansions of the
    different terms with respect to the variable $h$, we deduce that $U$ is a classical solution of
    \eqref{be-control}-\eqref{bid-control}. If $U$ is not smooth, this has to be done with
    test-functions and we obtain that $U$ is a viscosity solution of the problem.

    \noindent \textbf{(c)} Finally one proves a comparison result for \eqref{be-control}-\eqref{bid-control},
    which shows that $U$ is the unique viscosity solution of \eqref{be-control}-\eqref{bid-control}.
\end{proof}

We point out that, in this sketch of proof, the continuity (or uniform continuity) of $U$ is not as
crucial as it seems to be. Of course continuity can be obtained directly by working on the
definition of $U$ in this framework. But one may also show that $U$ is a discontinuous
viscosity solution (see Section~\ref{sect:stab}) and deduce continuity from the comparison result.
We insist on the fact that in this classical framework, people are mainly interested in 
cases where $U$ is continuous and therefore in assumptions ensuring this continuity.

Concerning Assumption \HBACPc, it is clear that $(iii)$ together with $(ii)$ ensure that for any
choice of control $\alpha(\cdot)$ there is a well-defined trajectory, by the Cauchy-Lipschitz
Theorem. Moreover, this trajectory $X(\cdot)$ exists for all times, thanks to the boundedness of
$b$. On the other hand, the boundedness of $l,c$ allows to show that  $U(x,t)$ is well-defined, 
bounded in $\R^N \times [0,\Tf]$ and even uniformly continuous there. Therefore we get all the
necessary information at the control level.

But Assumption \HBACPc plays also a key role at the pde level, in view of the comparison result:
indeed, it implies that the Hamiltonian $H$ satisfies the following property: for any $R\geq 1$

\noindent{\em There exists $M>0$, $C_1$ and a modulus of continuity $m: [0,+\infty) \to [0,+\infty)$
such that, for any $x,y\in \R^N$, $t,s \in [0,\Tf]$, $-R \leq r_1\leq r_2 \leq R \in \R$ and $p,q\in
\R^N$ 
$$
|H(x,t,r_1,p)-H(y,s,r_1,p)| \leq  \left(C_1|x-y|+m(|t-s|)\right)|p| 
    + m\left((|x-y|+|t-s|)R\right)\;,
$$
$$
H(x,t,r_2,p)-H(x,t,r_1,p)\geq -M(r_2-r_1)\; ,$$
$$
|H(x,t,r_1,p)-H(x,t,r_1,q)| \leq M|p-q|\;.$$
}
Of course, these properties are satisfied with $M=\max(||b||_\infty, ||c||_\infty,||l||_\infty)$ and
$m$ is the modulus of uniform continuity of $b,c,l$.

\

\section{Important remarks on the comparison proof}

We want to insist on several points here, and highlight several remarks that are important to
understand the methods and strategies we develop throughout this book.

\noindent\textbf{On proper Hamiltonians --}\label{page:c.positif} 
in the process of performing comparison between a subsolution $u$ and a supersolution $v$ (See
Section~\ref{sect:stab}), the initial step is to reduce the proof to the case when $r \mapsto
H(x,t,r,p)$ is increasing (or even non-decreasing) for any $x,t,p$. Such Hamiltonians are often
called ``proper''. 

This can be done through the classical change of unknown functions 
$$
    u(x,t) \to \tilde u(x,t) := u(x,t)\exp(-Kt)\;,
$$
and the same for $v\to\tilde v$, for some $K\geq M$. The Hamiltonian $H$ is changed into 
$$
    \tilde H(x,t,r,p):= \sup_{\alpha \in A}\,\left\{-b(x,t, \alpha)\exp(-Kt) \cdot p
    + [c(x,t, \alpha)+K] r-l(x,t, \alpha)\exp(-Kt)\right\}.
$$
This allows to reduce to the case where $c(x,t, \alpha) \geq 0$ for any $x,t,\alpha$, 
or even $\geq 1$. 

\keypoint{Note}{We will always assume in this book that, one way or the other, we can reduce to the
case when $c\geq 0$.}

\noindent\textbf{On the $x$ and $t$-dependence --}
the second point we want to emphasize is the $t$-dependence of $b$. It is well-know that, in the
comparison proof, the term 
$$\mathcal{Q}:=\left(C_1|x-y|+m(|t-s|)\right)|p| $$
is playing a key role. In order
to handle the difference in the behavior of $b$ in $x$ and $t$, one has to perform a proof with a
``doubling of variable'' technique which is different in $x$ and $t$. Namely we have to consider the
function
$$ 
    (x,t,y,s) \mapsto \tilde u(x,t)-\tilde v(y,s) - \frac{|x-y|^2}{\eps^2}-
    \frac{|t-s|^2}{\beta^2}-\eta (|x|^2+|y|^2)\;,
$$
where $0<\beta \ll \eps \ll 1$ and $0 <\eta \ll1$. We recall that the $\eta$-term ensures that this
function achieves its maximum while the $\eps,\beta$-terms ensure $(x,t)$ is close to $(y,s)$.
Therefore the maximum of this function is close to $\sup_{\R^N} (\tilde u -\tilde v)$.

The idea behind this different doubling in $x$ and $t$ is the following: the proof requires a 
quantity similar to $\mathcal{Q}$ above to be small. Now, since $|p|$ behaves like
$o(1)\eps^{-1}$, while $|x-y|$ is like $o(1)\eps$ and $|t-s|$ like $o(1)\beta$, the product
$C_1|x-y||p|$ is indeed small. But in order to ensure that the product $m(|t-s|)|p|$ is also small,
we need to choose $\beta$ small enough compared to $\eps$.

In this book, we want to handle cases when $b,c,l$ can be discontinuous on submanifolds
of $\R^N \times [0,\Tf]$. From a technical point of view, one quickly realizes that the $x$ and $t$ variables 
often play a similar role in this framework. 

\keypoint{Note}{Our assumptions on the behavior of $b,c,l$ or $H$ with respect to $x$ and $t$ will
essentially be the same.}

In particular, we will assume that $b$ is also Lipschitz continuous in $t$. This unnatural
hypothesis simplifies the proofs but we indicate in Section~\ref{sect:mgdt} how it can be removed at
the expense of more technicalities.

\noindent\textbf{On localization arguments --}
last but not least, this classical comparison proof does not use a real ``localization'' procedure. 
Of course, the role of the $-\eta (|x|^2+|y|^2)$-term is to ensure that the function associated to
the ``doubling of variable'' achieves its maximum. However, the way to play with the parameters, letting
first $\eta$ tend to $0$ and then sending $\beta$ and $\eps$ to zero afterwards implies that these
maximum points do not remain a priori bounded. 

\keypoint{Note}{In all the arguments in the book, we will use in a central way either the Lipschitz continuity or the
convexity of $H$ in $p$ in order to have a more local comparison proof.}

We systematically develop this point of view in Section~\ref{sect:htc}.

\section{Basic assumptions}\label{sec:BasicA}

The previous remarks lead us to replace \HBACPc by the following basic (yet less classical)
set of assumptions on the control problem:

\label{page:HBACP}\index{Control problem!basic}
\begin{assumption}{\HBACP}{Basic Assumptions on the Control Problem.}\\[-1.2cm]
\nopagebreak
\begin{enumerate}
\item[$(i)$] The function $\u0: \R^N\to \R$ is a bounded, continuous function.
\item[$(ii)$] The functions $b,c,l$ are bounded, continuous functions on $\R^N \times [0,\Tf] \times
A$ and the sets $(b,c,l)(x,t, A)$ are convex compact subsets of $\R^{N+2}$ for any $x\in \R^N$,
$t\in [0,\Tf]$ \footnote{The last part of this assumption which is not a loss of generality will be
used for the connections with the approach by differential inclusions.}.  \item[$(iii)$] For any
    ball $B\subset \R^N$, there exists a constant $C_1 (B)>0$ such that, for any $x,y \in \R^N$, $t
        \in [0,\Tf]$, $\alpha \in A$, we have
$$ |b(x,t, \alpha)-b(y,s,\alpha)| \leq C_1(B)\left( |x-y| +|t-s| \right)\; .$$
\end{enumerate} 
\end{assumption}

We will explain in Section~\ref{sect:mgdt} how to handle a more general dependence in time when the
framework allows it. In terms of equations and Hamiltonians, and although the following assumption
is not completely equivalent to \HBACP, we will use the

\label{page:BA-HJ}
\begin{assumption}{\HBAHJ}{Basic Assumptions on the Hamilton-Jacobi equation.}
There exists a constant $C_2>0$ and, for any ball $B\subset \R^N\times [0,\Tf]$, for any $R>0$, there
exists constants $C_1=C_1(B,R)>0, \gamma (R)\in \R $ and a modulus of continuity $m=m(B,R): [0,+\infty)
\to [0,+\infty)$ such that, for any $x,y \in B$, $t,s \in [0,\Tf]$, $-R \leq r_1 \leq r_2 \leq R$ and
$p,q \in \R^N$ $$ |H(x,t,r_1,p)-H(y,s,r_1,p)|\leq C_1[|x-y|+|t-s|]|p| + m(|x-y|+|t-s|)\;
,$$
$$ |H(x,t,r_1,p)-H(x,t,r_1,q)|\leq C_2 |p-q|\; ,$$
$$ H(x,t,r_2,p)-H(x,t,r_1,p) \geq \gamma(R) (r_2-r_1) \; .$$
\end{assumption}

In the next part ``Tools'', we introduce the key ingredients which allow to pass from the above
standard framework to the discontinuous one; they are concerned with 
\begin{enumerate}
    \item[a.] \textbf{Hamilton-Jacobi Equations:} we recall the notion of viscosity solutions and we
        revisit the comparison proof in order to have an easier generalization to the discontinuous
        case. We immediately point out that the regularization of sub and supersolutions by sup or
        inf-convolutions will play a more important role in the discontinuous setting than in the
        continuous one.  
    \item[b.] \textbf{Control problems:} the discontinuous framework leads to introduce
        \textbf{Differential inclusions} in order to define properly the dynamic, discount and cost
        when $b, c, l$ are discontinuous. We provide classical and less classical results on the DPP
        in this setting.  
    \item[c.] \textbf{Stratifications:} we describe the notion of Whitney stratification which is
        the notion used in Bressan and Hong \cite{BH} for the structure of the discontinuities of
        $H$ or the $(b,c,l)$ and we introduce the notions of ``Admissible Flat Stratification'',
        ``Locally Flattenable Stratification'', and ``Tangentially Flattenable Stratification'' which are useful
        for our approach.
\end{enumerate}

Using these tools requires to make some basic assumptions for each of them, which are introduced
progressively in this next part.  Apart from \HHJ and \HCP that we introduced above, we will use
\HBCL and \HST respectively for the Differential Inclusion and the Stratification.


\chapter{PDE Tools}
\fancyhead[CO]{HJ-Equations with Discontinuities: PDE Tools}
\label{chap:pde.tools}

\abstract{This chapter presents all the tools which involve only pde-type arguments: while
stability results, and in particular the ``half-relaxed limits method'', are just described, ``Strong
Comparison Results'' are revisited to obtain a version which can be used in the discontinuous
framework. Whitney stratifications are introduced and some of their properties are studied with the
regularization of subsolutions procedure in mind, a key step in the proof of comparison results for
stratified problems. The important notion of ``regularity of discontinuous functions'' is exposed.
Finally, properties of viscosity sub and supersolutions on the boundary are studied with two points
of view, linking their regularity and the approach of Lions-Souganidis for problems set on
networks.}

\section{Discontinuous viscosity solutions for equations with discontinuities,
    ``half-relaxed limits'' method}
\label{sect:stab}

In this section, we recall the classical definition of discontinuous viscosity solutions introduced
by Ishii\cite{I1} for equations which present discontinuities. We have chosen to present it in the
first-order framework since, in this book, we are mainly interested in Hamilton-Jacobi Equations but
it extends without major changes to the case of fully nonlinear elliptic and parabolic pdes. We
refer to the Users' guide of Crandall, Ishii and Lions \cite{Users}, the books of Bardi and
Capuzzo-Dolcetta \cite{BCD} and Fleming and Soner \cite{FS} and  the CIME courses \cite{BCESS,CIME}
for more detailed presentations of the notion of viscosity solutions in this more general setting.

We (unavoidably) complement this definition by the description of the discontinuous stability
result, often called ``Half-Relaxed Limits Method'', being clearly needed when dealing with
discontinuities. We recall that it allows passage to the limit in fully nonlinear elliptic and
parabolic pdes with just an $L^\infty$--bound on the solutions. The ``Half-Relaxed Limits Method''
was introduced by Perthame and the first author in \cite{BP1} and developed in a series of works
\cite{BP2,BP3}. One of its first striking consequences was the ``Perron's method'' of Ishii
\cite{IPer}, proving the existence of viscosity solutions for a very large class of first- and
second-order equations (see also the above  references for a complete presentation).

\label{not:semic}\label{not:usc}
The definition of viscosity solutions uses the upper semicontinuous (\usc) envelope and lower
semicontinuous (\lsc) envelope of both the (sub and super) solutions and of the Hamiltonians and we
introduce the following notations: if $f:A\subset\R^p\to\R$ is a locally bounded function (possibly
discontinuous), we denote by $f^*$ its \usc envelope\index{Lower semicontinuous, upper semicontinuous envelopes}
$$ f^* (X) = \limsup_{\tilde X\to X}\; f(\tilde X) \quad \hbox{for  }X \in A\;, $$
and by $f_*$ its \lsc envelope
$$ f_* (X) = \liminf_{\tilde X\to X}\; f(\tilde X)\quad \hbox{for  }X \in A\;.$$

Throughout this section, we use $X\in\R^N$ as the generic variable to cover both the stationary and
evolution cases where respectively, $X=x\in\R^n$ or $X=(x,t)\in\R^{n}\times\R$.

\subsection{Discontinuous viscosity solutions}
\label{subsec.disc.visc.sol}
\index{Viscosity solutions!Ishii definition}
We consider a generic Hamiltonian $\G: \OOb \times \R \times \R^N \to \R$ where $\OO$ is an open
subset of $\R^N$ and $\OOb$ denotes its closure. We just assume that $\G$ is a locally bounded
function which is defined pointwise.

The definition of viscosity sub and supersolution is the following 
\begin{definition}\emph{--- Discontinuous Viscosity Solutions.}
    \label{class.visc. sol}\smsp
A locally bounded function $u:\OOb \to \R$ is a
viscosity subsolution of the equation
\begin{equation}\label{Geqn}
\G(X,u,Du) = 0 \quad \hbox{ on }\OOb
\end{equation}
if, for any $\varphi \in C^1(\overline\OO)$, at a maximum point $X_0 \in\OOb$
of $u^*-\varphi$, one has
$$\G_*(X_0,u^*(X_0),D\varphi(X_0))\leq 0\; .$$

A locally bounded function $v:\OOb \to \R$ is a
viscosity supersolution of Equation (\ref{Geqn}) if,
for any $\varphi \in C^1(\overline\OO)$, at a minimum point $X_0 \in\OOb$
of $v_*-\varphi$, one has
$$\G^*(X_0,v_*(X_0),D\varphi(X_0))\geq 0\; .$$

A (discontinuous) solution is a function which is both viscosity sub and supersolution of the
equation.
\end{definition}

Several classical remarks on this definition: 

\noindent$(i)$ In general, the notion of subsolution is given for \usc functions while the
notion of super-solution is given for \lsc functions: this may appear natural when looking at the above
definition where just $u^*$ and $v_*$ play a role and actually we can reformulate the above definition for
general functions as: $u$ is a subsolution if and only if the \usc function $u^*$ is a subsolution and $v$ is a supersolution
if and only if the \lsc function $v_*$ is a supersolution.  The interest of this more general definition comes from the applications,
for example to control problems, where we face functions which are a priori neither \usc nor \lsc and still we wish to prove
that they are sub and supersolution of some equations.  Therefore such a formulation is needed. But when we will have to give a
result which holds for subsolutions (or supersolutions), we will assume the subsolution to be \usc (or the supersolution to be \lsc) in order to lighten the notations in the statement.

\noindent $(ii)$ If the space of ``test-functions'' $\varphi$
which is here $C^1(\overline\OO)$ is changed into $C^2(\overline\OO)$, $C^k(\overline\OO)$ for any
$k>1$ or $C^\infty(\overline\OO)$, we obtain an equivalent definition. Then, for a classical
stationary equation (say in $\R^n$) like $$ H(x,u,Du)=0 \quad \hbox{in  }\R^n\, ,$$
the variable $X$ is just $x$, $N=n$ and $Du$ stand for the usual gradient of $u$ in $\R^n$. 
But this framework also contains the case of evolution equations
$$ u_t + H(x,t,u,D_x u)=0 \quad \hbox{in  }\R^n\times (0,\Tf)\, ,$$
where $X=(x,t) \in \R^n\times (0,\Tf)$, $N=n+1$ and $Du=(D_x u,u_t)$ where $u_t$ denotes the
time-derivative of $u$ and $D_x u$ is the derivative with respect to the space variables $x$, and
the Hamiltonian reads $$ G(X,r,P)=p_t + H(x,t,r,p_x)\; ,$$
for any $(x,t) \in \R^n\times (0,\Tf)$, $r\in \R$ and $P=(p_x,p_t)$.

\noindent $(iii)$ 
This definition is a little bit strange since the equation is set on a closed
subset, a very unusual situation.  There are two reasons for introducing it this way: the first one
is to unify equation and boundary condition in the same formulation as we will see below. With such
a general formulation, we avoid to have a different results for each type of boundary conditions.
The second one, which provides also a justification of the ``boundary conditions in the viscosity
sense'' is the convergence result we present in the next section.

To be more specific, let us consider the problem
$$
\left\{
\begin{array}{cc}
F(x,u,Du) = 0 & \hbox{in  }\OO \subset \R^n ,\cr
L(x,u,Du) = 0 & \hbox{on  }\partial \OO,\cr
\end{array}\right.
$$ where $F,L$ are given continuous functions. 
If we introduce the function $G$ defined by
$$\G(x, r, p)  = \left\{ \begin{array}{cc}F(x,r,p) & \hbox{if   }x\in \OO
 , \cr L(x,r,p) & \hbox{if  }x\in  \partial \OO.\cr\end{array}\right.$$
we can just rewrite the above problem as
$$\G(x,u,Du) = 0 \quad \hbox{ on } {\overline \OO}\;,$$
where the first important remark is that $\G$ is a priori a discontinuous Hamiltonian.
Hence, even if we assume $F$ and $L$ to be continuous, we face a typical example which we want to
treat in this book!

The interpretation of this new problem can be done by setting the
equation in $\overline \OO$ instead of $\OO$. 
Applying blindly the definition, we see that $u$ is a subsolution if
$\G_*(x,u^*,Du^*) \leq 0$ on $\overline \OO$,
i.e. if
$$\begin{cases}
    F(x,u^*,Du^*) \leq 0 & \hbox{in  }\OO\;,\\[2mm] \min ( F(x,u^*,Du^*) , L(x,u^*,Du^*))
\leq 0 & \hbox{on  }\partial \OO\;,\end{cases}
$$
while $v$ is a supersolution if $\G^*(x,v_*,Dv_*) \geq  0$ on $\overline \OO$,
$i.e.$ if
$$\begin{cases}
    F(x,v_*,Dv_*) \geq 0 & \hbox{in  }\OO\;,\\[2mm] \max ( F(x,v_*,Dv_*) , L (x,v_*,Dv_*)
) \geq 0 & \hbox{on  }\partial \OO\;.\end{cases}
$$
Indeed, we have just to compute $\G_*$ and $\G^*$ on $\OOb$ and this is where the ``$\min$'' and the
``$\max$'' come from on $\partial \OO$.

Of course, these properties have to be justified and this can be done by the discontinuous stability
result of the next section which can be applied for example to the most classical way to solve the
above problem, namely the vanishing viscosity method 
$$\begin{cases}
-\e \Delta
    u_\e + F(x,u_\e ,Du_\e) = 0 & \hbox{ in  }\OO\;,\\
    \qquad L (x,u_\e,Du_\e) = 0  &\hbox{ on  } \partial \OO\;.
\end{cases}$$
Indeed, by adding a $-\e \Delta $ term, we regularize the equation
in the sense that one can expect to have more regular solutions for this
approximate problem---typically in $C^2(\OO)\cap C^1(\overline\OO)$.

To complete this section, we turn to a key example: the case of a two half-spaces problem, 
which presents a discontinuity along an hyperplane. We use the following framework: 
in $\R^N$, we set  $\Omega_1=\{x_N>0\}$, $\Omega_2=\{x_N<0\}$ and $\mathcal{H}=\{x_N=0\}$.
We assume that we are given three continuous Hamiltonians, $H_1$ on $\overline \Omega_1$, 
$H_2$ on $\overline  \Omega_2$ and $H_0$ on $\mathcal{H}$. Here, $X=(x,t)$ and let us introduce 
$$\G(X,r,p):=\begin{cases}
	p_t+H_1(x,t,r,p_x) & \text{if }x\in\Omega_1\;,\\
	p_t+H_2(x,t,r,p_x) & \text{if }x\in\Omega_2\;,\\
	p_t+H_0(x,t,r,p_x) & \text{if }x\in\mathcal{H}\;.\\
\end{cases}$$
Then solving $\G(X,u,Du)=0$ for $X=(x,t)\in\R^{N+1}$ means to solve the equations
$u_t+H_i(x,t,u,Du)=0$ in each $\Omega_i$ ($i=1,2$) with the ``natural'' conditions on $\mathcal{H}$
given by the Ishii's conditions for the sub and super-solutions, namely
$$ \begin{cases}\min(u_t+H_1 (x,t,u^*,Du^*),u_t+H_2(x,t,u^*,Du^*),u_t+H_0(x,t,u^*,Du^*)) \leq 0 & 
    \hbox{on  }\mathcal{H} \;,\\
\max(u_t+H_1(x,t,v_*,Dv_*),u_t+H_2(x,t,v_*,Dv_*),u_t+H_0(x,t,v_*,Dv_*)) \geq 0&
\hbox{on  }\mathcal{H}\;.\end{cases}$$

\begin{remark}We have decided to present the definition of viscosity solution on a closed space
    $\OOb$ for the reasons we explained above.  But we can define as well equations set in open
    subset of $\R^N$ (typically $\OO$)  or open subsets of $\OOb$ 
(typically $\OOb\cap B(X,r)$ for some $X \in \OOb$ and $r>0$). 
The definition is readily the same, considering local maximum points of $u^*-\varphi$ 
    or minimum points of $v_*-\varphi$ which are in $\OO$ or $\OOb\cap B(X,r)$.
\end{remark}

We end this section with a classical ``trick'' that is used in many stability results like the
half-relaxed limits method, which is detailed in the next section.
\begin{lemma}\label{lem:strict.point} When testing the sub or supersolution condition for an equation of the type 
$\G(X,u,Du)=0$, if $u-\varphi$ reaches a local extremum at $X_0$, we can always assume that $X_0$ is a strict maximum or minimum point, without changing $D\varphi(X_0)$.
\end{lemma}

We point out that a immediate consequence of this lemma is that we have an equivalent definition of viscosity
sub and supersolutions by considering only strict local maximum/minimum points.

\begin{proof}
    In the case of a maximum point, we just need to replace $\varphi$ by
    $\psi(X):=\varphi(X)-c|X-X_0|^2$ where $c>0$: it is clear that $u-\psi$ has a strict maximum at
    $X_0$ and moreover since $D\varphi(X_0)=D\psi(X_0)$, the subsolution condition still takes the
    form $$\G(X_0,u(X_0),D\varphi(X_0))\leq0\;.$$ Of course, the same argument applies for the supersolution
    condition by adding this time $c|X-X_0|^2$ to $\varphi$.
\end{proof}

Notice that the same trick works for second-order equations, but in order to keep the second-order
derivatives unchanged we have to use $\varphi(X)\pm c|X-X_0|^4$.

\subsection{The half-relaxed limits method}\label{sec:hrl}

\newcommand{\fe}{f_\eps}

In order to state it we use the following notations: if $A\subset \R^p$ and if $(\fe)_\e$ is a
sequence of uniformly locally bounded real-valued functions defined on $A$, the half-relaxed limits
of $(\fe)_\e$ are defined, for any $X \in A$, by\index{Half-relaxed limits}
$$\limssup \fe (X)\;=\;\limsup_\YtoXandEPStoZERO
 \fe(Y)\quad \hbox{and}\quad\limiinf \fe (X)\;=\;\liminf_\YtoXandEPStoZERO
  \fe(Y)\; .$$

\begin{theorem}\label{hrl}\emph{--- Half-relaxed limits.}\smsp
    Assume that, for $\e >0$, $u_{\e}$ is a
viscosity subsolution \resp{a supersolution} of the equation
$$\Ge (X,u_{\e},Du_{\e}) =0 \quad \hbox { on } \overline \OO\; ,
$$
where $(\Ge )_\e$ is a sequence of uniformly locally bounded
functions in $ \overline \OO \times \R
\times \R^N $. If the
functions $u_{\e}$ are uniformly locally bounded on $ \overline \OO$, then
$\overline u=\limssup u_{\e}$ \resp{$\underline u=\limiinf u_{\e}$} 
is a subsolution \resp{a supersolution}
of the equation $$\underline \G (X,u,Du) =0 \quad \hbox { on }
\overline \OO\; ,$$ where $\underline \G=\limiinf \Ge$.
\resp{of the equation
$$\overline \G (X,u,Du) =0 \quad \hbox { on }  \overline  \OO\; ,$$
where
$\overline \G=\limssup \Ge$}.
\end{theorem}

In order to compare them, we recall that the first stability result for viscosity solutions is given
in the introductory article of Crandall and Lions \cite{CL}: it takes the form
\begin{theorem}\label{stab:first} Assume that, for $\e >0$, $u_{\e}\in C(\OO)$ is a
viscosity subsolution \resp{a supersolution} of the equation
$$\Ge (X,u_{\e},Du_{\e}) =0 \quad \hbox { in } \OO\; ,
$$
where $(\Ge )_\e$ is a sequence of continuous
functions in $\OO \times \R
\times \R^N $. If $u_{\e} \to u$ in $C(\OO)$ and if $\Ge \to \G$ in $C(\OO \times \R
\times \R^N )$, then $u$
is a subsolution \resp{a supersolution}
of the equation $$ \G (X,u,Du) =0 \quad \hbox { in }
 \OO\; .$$
\end{theorem}

We recall that the convergence in the space of continuous functions ($C(\OO)$ or $C(\OO \times \R
\times \R^N) $) is the local uniform convergence.

Theorem~\ref{stab:first} is, in fact, a particular case of Theorem~\ref{hrl}. Indeed, as the proof
will show, the result of Theorem~\ref{hrl} remains valid if we replace $\overline  \OO$ by $\OO$
and if $\ue$ and $\Ge$ converge uniformly then $u=\overline u=\underline u$ and $\G=\overline
\G=\underline \G$.

Hence Theorem~\ref{hrl} is more general when applied to either sub or supersolutions: its main interest
is to allow the passage to the limit in the notion of sub and supersolutions with very weak assumptions on the
solutions but also on the equations: only uniform local $L^\infty$--bounds. In particular, phenomenas like {\em
boundary layers} can be handled with such a result. This is a striking difference with
Theorem~\ref{stab:first} which, in practical uses, requires some compactness of the $\ue$'s in the
space of continuous functions (typically some gradient bounds) in order to have a converging
subsequence.
 
The counterpart is that we do not have a limit anymore, but two half-limits $\overline u$ and
$\underline u$ which have to be connected in order to obtain a real convergence result. In fact, the
complete {\bf Half-Relaxed Limit Method} is performed as follows \index{Comparison
result!strong}\label{not:SCR}
\begin{enumerate}
\item Get a locally (or globally) uniform $L^\infty$--bound for the $(\ue)_\e$.
\item Apply the above discontinuous stability result.
\item The inequality $\uu \leq \ou$ on  $\overline \OO$ holds by definition.
\item To obtain the converse inequality, use a {\bf Strong Comparison Result}, \SCR in short,
 \ie a comparison result which is valid for {\em discontinuous} sub and supersolutions, which yields
$$ \ou \leq \uu \quad \hbox{in }\;\OO\;(\hbox{or on  }\;\overline \OO\,)\;. $$
\item From the \SCR, we deduce that $\ou=\uu$ in $\OO$ (or on  $\overline \OO$). 
Setting $u:=\ou=\uu$, it follows that $u$ is continuous (because $\ou$ is \usc and $\uu$ is \lsc)
        and it is easy to show that, $u$ is {\em the unique solution} of the limit equation, by
        using again the \SCR. 
\item Finally, we also get the convergence
of $u_{\e}$ to $u$ in $C(\OO)$ (or in $C(\overline \OO)$) (see Lemma~\ref{ubegalub} below).
\end{enumerate}

It is clear that, in this method, \SCR play a central role and one of the main challenge in this
book is to show how to obtain them in various contexts.

Now we give the \noindent{\bf Proof of Theorem~\ref{hrl}}. 
We do it only for the subsolution case, the supersolution one being analogous.

We first remark that $\limssup u_{\e} = \limssup u_{\e}^*$ and therefore changing $u_{\e}$ in
$u_{\e}^*$, we can assume without loss of generality that $\ue$ is \usc.  Recall also
that by Lemma~\ref{lem:strict.point}, we are always reduced to consider strict extremum points in viscosity
inequalities testing. The proof is based on the
\begin{lemma}\label{convptsmaxdisc} Let
$(w_\e)_\e$ be a sequence of uniformly bounded \usc functions on $\OOb$ and 
$\ow = \limssup w_\e$. If $X \in \OOb$ is a strict local maximum point 
of $\ow$ on $\OOb$, there exists a subsequence
$(w_{\e'})_{\e'}$ of $(w_\e)_{\e}$ and a sequence
$(X_{\e'})_{\e'}$ of points in $\OOb$ such that,
for all $\e'$, $X_{\e'}$ is a local maximum point of $w_{\e'}$ in $\OOb$, the sequence
$(X_{\e'})_{\e'}$ converges to $X$ and
$w_{\e'}(X_{\e'})\to \ow(X)$.
\end{lemma}

We first prove Theorem~\ref{hrl} by using the lemma. Let $\varphi \in C^1(\OOb)$ and let $X\in \OOb$ 
be a strict local maximum point  de $\overline u -\varphi$. We apply Lemma~\ref{convptsmaxdisc} 
to $w_\e = u_{\e} -\varphi$ and $\ow=\overline u-\varphi =  \limssup\,(u_{\e}-\varphi)$.
There exists a subsequence $(u_{\e'})_{\e'}$ and a sequence $(X_{\e'})_{\e'}$ such that, for all 
$\e'$, $X_{\e'}$ is a local maximum point of $u_{\e'}-\varphi$ on $\OOb$. But $u_{\e'}$ is a
subsolution of the $\G_{\e'}$-equation, therefore
$$\G_{\e'} (X_{\e'},u_{\e'}(X_{\e'}),D\varphi (X_{\e'})) \leq 0\; .$$
Since $X_{\e'}\to X$ and since $\varphi$ is smooth $D\varphi (X_{\e'}) \to D\varphi(X)$; 
but we have also $u_{\e'}(X_{\e'})\to \overline u(X)$, therefore by definition of  $\underline \G$
$$\underline \G (X,\overline u(X),D\varphi (X)) 
\leq \liminf\, \G_{\e'} (X_{\e'},u_{\e'}(X_{\e'}),D\varphi (X_{\e'})) \;.$$
This immediately yields
$$ \underline \G (X, \overline u(X),D\varphi (X)) \leq 0\; ,$$
and the proof is complete.

\begin{proof}[Proof of Lemma~\ref{convptsmaxdisc}] 
    Since $X$ is a strict local maximum point of $ \ow$ on $\OOb$, there exists $r>0$ such that
$$\forall Y \in \OOb \cap \overline B(X,r)\;,\quad  \ow (Y)
\leq  \ow (X)\; ,$$
the inequality being strict for $Y\neq X$. But $\OOb \cap \overline 
B(X,r)$ is compact and $w_\e$ is \usc, therefore, for all $\e >0$, 
there exists a maximum point $X_\e$ of $w_\e$ on $\OOb \cap \overline  B(X,r)$. In other words
\begin{equation}\label{pourlimsup}
\forall Y\in \OOb \cap \overline 
B(X,r)\;,\quad  w_\e (Y) \leq w_\e(X_\e)\; 
.\end{equation}
Now we take the $\limsup$ as $Y\to X$ and $\e \to 0$: we obtain
    $$  \ow (X) \leq \limsup_{\e\to0} \, w_\e(X_\e)\; .$$

Next we consider the right-hand side of this inequality: 
extracting a subsequence denoted by $\e'$, we have 
$\limsup_\e \, w_\e(X_\e) = \lim_{\e'}w_{\e'} (X_{\e'})$ and since $\OOb  \cap \overline B(X,r)$
is compact, we may also assume that $X_{\e'} \to \bar X \in \OOb  \cap \overline B(X,r)$. 
But using again the definition of the $\limssup$ at $\bar X$, we get
    $$  \ow (X) \leq \limsup_{\e\to0} \, w_\e(X_\e)= 
    \lim_{\e'\to0}w_{\e'} (X_{\e'}) \leq  \ow(\bar X) \; .$$
Since $X$ is a strict maximum point of $ \ow$ in $\OOb  \cap \overline
B(X,r)$ and that $\bar X \in \OOb \cap \overline B(X,r)$, 
this inequality implies that $\bar X=X$ and that $w_{\e'} (X_{\e'}) \to  \ow (X)$, so that
    the proof is complete.
\end{proof}

Controlling the liminf and limsup also implies local uniform convergence:
\begin{lemma}\label{ubegalub} If $\KK$ is a compact subset of $\OOb$ and if $\ou=\uu$ 
on $\KK$ then $u_{\e}$ converges uniformly to the function $u:=\ou=\uu$ on $\KK$.
\end{lemma}

\noindent{\bf Proof of Lemma~\ref{ubegalub} :} Since $\ou=\uu$ on $\KK$ and since $\ou$ is \usc and
$\uu$ is \lsc on $\OOb$, $u$ is continuous on $\KK$.  We first consider 
$$M_\e  =\sup_{\KK}\,\left(u_{\e}^* - u \right)\;.$$ 
The function $u_{\e}^*$ being \usc and $u$ being continuous, this supremum is in fact a maximum and
is achieved at a point $X_\e$. The sequence $\left(u_{\e}\right)_\e$ being locally uniformly bounded,
the sequence $\left(M_\e\right)_\e$ is also bounded and, $\KK$ being compact, we can extract
subsequences such that $M_{\e'} \to \limsup_\e\,M_\e$ and 
$X_{\e'} \to \bar X \in \KK$. But by the definition of the $\limssup$, 
$\limsup u_{\e'}^*(X_{\e'}) \leq \ou (\bar X)$ while we have also $u(X_{\e'}) \to u(\bar X)$ 
by the continuity of $u$. We conclude that
$$ \limsup_{\e\to0}\,M_\e = \lim_{\e'\to0}\, M_{\e'} = 
\lim_{\e'\to0}\, (u_{\e'}^*(X_{\e'}) -u(X_{\e'}) ) \leq \ou (\bar X)-u(\bar X)=0\; .$$
This part of the proof gives half of the uniform convergence, 
the other part being obtained analogously by considering 
$\displaystyle{\tilde M_\e  =\sup_{\KK}\,\left(u -(u_{\e})_* \right)}$.

\section{Strong comparison results: how to cook them?}
\label{sect:htc}

\index{Comparison result!strong}\index{Comparison result!global}
\index{Comparison result!local}\label{not:SCR2}
In the previous section, we have seen that \SCR are key tools which are needed to use the
``Half-Relaxed Limit Method''. We have used the terminology ``strong'' because such comparison
results have to hold for discontinuous sub and supersolutions, which are only \usc and \lsc
respectively. From a technical point of view, it is easier to compare at least continuous
sub and supersolutions and of course, some comparison results may even fail in the discontinuous
framework. However, in this book we mainly prove \SCR therefore the expression ``comparison
result'' always refers to a \emph{strong} one.

In general, a comparison result is a \emph{global} inequality (i.e. on the whole domain) between
sub and supersolutions. However, in the case of Hamilton-Jacobi Equations with discontinuities it
is far easier, if not necessary, to argue locally. This is why in this section we explain how to
reduce the proof of {\em global} comparison results to the proof of {\em local} comparison results.
We do not pretend this section to cover all cases but we have tried to make it as general as we could.

\index{Viscosity solutions!comparison results}

\subsection{Stationary equations}

In this section we are in the situation where $X=x$ is the space variable in $\R^N$ or a subset of
it, and no time variable is involved here. We consider a general equation
\begin{equation}\label{GenHJ}
\G(x,u,Du)= 0\quad \hbox{on  }\mF \; ,
\end{equation}
where $\mF$ is a closed subset of $\R^N$ and $\G$ is a continuous or discontinuous 
function on $\mF \times \R \times \R^N$. 

\label{not:USCS}\label{not:LSCS}
We introduce the following notations: $\USCS(\mF)$ is a subset of \usc subsolutions of \eqref{GenHJ}
while $\LSCS(\mF)$ is a subset of \lsc supersolutions of \eqref{GenHJ}. We prefer to remain a little
bit vague on these subsets but the reader may have in mind that they are generally defined by some
growth conditions at infinity if $\mF$ is an unbounded subset of $\R^N$. In these definitions, we
may replace below $\mF$ by a subset (open or closed) of~$\mF$ and we use below the following
notations 
$$ \mF^{x,r} := B(x,r)\cap \mF \quad \hbox{and}\quad  \partial \mF^{x,r} := \partial
{B(x,r)}\cap \mF\; .$$
Finally we denote by $\USCS(\mF^{x,r})$ \resp {$\LSCS(\mF^{x,r})$} the set of \usc \resp{\lsc}
functions on $\overline{\mF^{x,r}}$ which are subsolutions \resp{supersolutions} of $\G=0$ in 
$\mF^{x,r}$. Notice that, for these sub and supersolutions, no viscosity inequality is imposed on
$\partial {B(x,r)}$.

By ``global'' and ``local'' comparison results we mean the following

\index{Comparison result!global}\label{not:GCR}\label{page:GCR}
\begin{assumption}{$\GCR^\mF$}{Global Comparison Result in $\mF$.}
 For any $u \in \USCS(\mF)$ and $v \in \LSCS(\mF)$, we have $u\leq v$ on $\mF$\;.
\end{assumption}

\index{Comparison result!local}\label{not:LCR}\label{page:LCR}
\begin{assumption}{$\LCR^\mF$}{Local Comparison Result in $\mF$.}
For any $x \in \mF$, there exists $\bar r>0$ such that, 
if $u \in \USCS(\mF^{x,\bar r})$, $v \in \LSCS(\mF^{x,\bar r})$ then for any $0<r\leq \bar r$,
$$ \max_{\overline{\mF^{x,r}}}(u-v)_+ \leq \max_{\partial \mF^{x,r} }(u-v)_+ .$$
\end{assumption}

 \label{pos.neg} We recall that $s_+$ denotes the positive part of $s\in \R$ while $s_-$ stands for its negative part
 and we extend this definition to functions.

Writing the $\LCR^\mF$ with the $(...)_+$ inequality is quite standard, see for instance \cite{GT}.
The meaning of this formulation is the following: either $u\leq v$ in $\overline{\mF^{x,r}}$ and we
are done locally speaking, or the maximum of $u-v$ is positive, but controlled by the values
at the boundary. The reader may be surprised by the formulation of $\LCR^\mF$ with both
$\bar r$ and $r$ but, on one hand, proving such result for all $r$ small enough (instead of a fixed
$r>0$) turns out to be the same in general. On the other hand, this formulation will give 
us more flexibility while presenting the strategy to reduce the proof of $\GCR^\mF$ to $\LCR^\mF$.

In the rest of this section, we skip the reference to $\mF$ in \LCR and \GCR since there is no
ambiguity here.  It is clear that proving \LCR seems much easier because of the compactness of
$\overline {\mF^{x,r}}$ and the fact that we only have to prove them for $0<r\leq\bar r= \bar r(x)$.

Indeed, the behavior at infinity of $u$ and $v$ does not play a role anymore and moreover we only
use local properties of $\G$; in particular, if $\G$ has discontinuities which form a
stratification, we can use this localization to restrict to a ball where some part of the stratification is flat with
suitable properties nearby, \cf~Section~\ref{sect:whitney}.

Now we formulate a first key assumption in order to reduce \GCR to \LCR.

\label{page:LOCa}
\begin{assumption}{\LOCa}{Localization assumption one.}
    If $\mF$ is unbounded, for any $u \in \USCS(\mF)$, for any $v \in \LSCS(\mF)$, there exists a
    sequence $ (u_\alpha)_{\alpha>0}$ of \usc subsolutions of \eqref{GenHJ} such that $u_\alpha(x) -
    v (x) \to -\infty$ when $|x|\to +\infty$, $x\in \mF$.  Moreover, for any $x\in \mF$,
    $u_\alpha(x) \to u(x)$ when $\alpha \to 0$.
\end{assumption}

In the above assumption, we do not write that $u_\alpha \in \USCS(\mF)$ because this is not the case
in general: typically, $\USCS(\mF)$ may be the set of {\em bounded} subsolutions of \eqref{GenHJ}
while $u_\alpha$ is not expected to be bounded.

The main consequence of \LOCa is that there exists $\xb \in \mF$ such that
$$u_\alpha(\xb) - v (\xb)= \max_{\mF}(u_\alpha-v)\; ,$$ 
and the basic ideas of the reduction of \GCR to \LCR can be understood through
the two following particular cases which we will generalize afterwards.

\noindent\emph{$(i)$ Strict local maximum point ---}
If $\xb$ is a {\em strict} local maximum point of $u_\alpha-v$ for any $\alpha>0$ small enough and $\bar
r:=\bar r(\xb)$ is defined in $\LCR^\mF$, then for any $0<r<\bar r$,
$$ (u_\alpha-v)_+ (\xb) \leq \max_{\partial \mF^{x,r} }(u_\alpha-v)_+\;.$$
But on the other hand, the {\em strict} local maximum point property implies that, if $r$ is small enough
$$ (u_\alpha-v) (\xb) > \max_{\partial \mF^{x,r} }(u_\alpha-v)\;.$$
So, if $(u_\alpha-v) (\xb) >0$, these two inequalities lead to contradiction and therefore, we
necessarily have $(u_\alpha-v) (\xb) \leq 0$ which implies $u_\alpha \leq v$ in
$\mF$. Since this is true for any $\alpha>0$ small enough, we let $\alpha$ tend to $0$ to conclude
that $u\leq v$, \ie\ \GCR holds.

Of course, this first case, although being rather illuminating, seems unrealistic. Indeed, after the
standard localization argument producing $u_\alpha$ that the reader may have in mind---or see how we
check \LOCa below---, it is clearly impossible in general to show that $u_\alpha-v$ has at least a
{\em strict} local maximum point, or to build $u_\alpha$ in order that this property holds. A second
argument is needed to possibly transform a local maximum point into a strict local maximum point, or
to be able to perform a similar proof as above in order to obtain \GCR. 

\

\noindent \emph{$(ii)$ Strict subsolution ---} This second case is more realistic: let
us assume that $u_\alpha$ is a {\em strict} subsolution, \ie there exists $\eta(\alpha)>0$ such that
$$\G(x,u_\alpha,Du_\alpha )\leq -\eta(\alpha)<0\quad \hbox{on  }\mF \; ,$$
and that $(r,p_x)\mapsto \G(x,r,p_x)$ is uniformly continuous in $\R\times \R^N$, uniformly \wrt $x$.
In this case, for $0<\delta \ll 1$ we set
$$ u^\delta_\alpha (x):=u_\alpha (x)-\delta |x-\xb|^2\; ,$$
where $\xb$ is defined as above. Thanks to the assumptions on $u_\alpha$ and $\G$, if $\delta$ is
chosen small enough, we see that, for any $\bar r$, $u^\delta_\alpha$ is a subsolution in $\mF^{x,\bar r}$, and
moreover $\xb$ is a {\em strict} maximum point of $u^\delta_\alpha-v$ in $\mF^{x,\bar r}$.
Therefore we are in an analogous situation as in the first case, 
$(u^\delta_\alpha-v)(\xb)=(u_\alpha-v)(\xb)=0$ and we conclude in the same way.

Our aim is to present a generalization of these two particular cases, especially the
second one. As the reader will notice, in the two main frameworks we investigate below---the 
``Lipschitz case'' and the ``convex case''---, only the convex framework will be really different
from case $(ii)$ above; in the ``Lipschitz case'', we will only formulate differently the
arguments.

\

In order to introduce the second localization hypothesis, let us define
$$\floorF{f}:=f(x)-\max_{y\in\partial\mF^{x,r}}f(y)\;,$$
which in some sense measures the variation of $f$ between $x$ and the boundary. Notice that since
$\max(f+g)\leq\max(f)+\max(g)$, this operator enjoys the following property
\begin{equation}\label{sur-add}
\floorF{f}+\floorF{g}\leq\floorF{f+g}\;.
\end{equation}

\label{page:LOCb}
\begin{assumption}{\LOCb}{Localization assumption two.}
    For any $x \in \mF$, $r>0$, if $u \in \USCS(\mF^{x,r})$, there exists a sequence
    $(u^\delta)_{\delta>0}$ of functions in $\USCS(\mF^{x,r})$ such that $\floorF{u^\delta -
    u}\geq\eta(\delta)>0$ for any $\delta$.  Moreover, for any $y\in \mF^{x,r}$, $u^\delta(y) \to
    u(y)$ when $\delta \to 0$.
\end{assumption}

As we have already used it in the study of the two particular cases above the role of \LOCa is clear:
\LOCa leads to a standard localization procedure. Instead of having to prove the comparison in $\mF$ which
can be unbounded, it allows to do it only on a compact subset of $\mF$. This has several advantages: first,
we can consider maximum points for the \usc function $u_\alpha-v$ in such a compact subset,  while
this is not, in general, the case for $u-v$ in $\mF$ since $u,v$ can also be unbounded. But, reducing the proof to \LCR,
we can also have more general assumptions on $\G$: the reader may compare \HBACPc with \HBACP and/or \HBAHJ in Section~\ref{sec:BasicA}.

The role of \LOCb is to give a suitable replacement of the construction of $u^\delta_\alpha$ in the
second particular case we describe. It is is a technical assumption which allows to make sure that
in \LCR the max is not attained at the boundary, by replacing $u$ with another subsolution which has
a greater variation between $x$ and the boundary. This is a key point in the proof of the main
result that we give now.

\begin{proposition}\label{reducCR}\emph{--- Reduction to a Local Comparison Result.}\smsp
 Assuming \LOCa and \LOCb, $\LCR$ implies $\GCR$.
\end{proposition}

\begin{proof}Given $u \in \USCS(\mF)$ and $v \in \LSCS(\mF)$, we have to prove that $u\leq v$ on $\mF$. 

Instead of comparing $u$ and $v$, we are going to compare $u_\alpha$ and $v$ for $u_\alpha$ given by
    \LOCa and then to let $\alpha$ tend to $0$. Arguing in that way and droping the $\alpha$ for
    simplifying the notations means that we can assume without loss of generality that $u(x)-v(x)
    \to -\infty$ when $|x|\to +\infty$, $x\in \mF$ and therefore we can consider $M:=\max_\mF (u-v)$
    and we argue by contradiction, assuming that $M>0$.

    Since $\mF$ is closed, $u-v$ is \usc and tends to $-\infty$ at infinity, this function achieves
    its maximum at some point $x\in \mF$. Considering $r>0$ for which \LCR holds, this means that $\floorF{u-v}\geq0$.

    Now we apply \LOCb. Since
    $u^\delta \in \USCS(\mF^{x,r})$ and \LCR holds, we get the following alternative 

    \noindent \textbf{(a)}
    either $u^\delta \leq v$ in $\overline{\mF^{x,r}}$, but this cannot be the case for $\delta$
    small enough since $u^\delta(x)- v(x)\to u(x)-v(x)>0$;

    \noindent \textbf{(b)} or $\max_{\overline{\mF^{x,r}}}(u^\delta -v)>0$ and $$
    \max_{\overline{\mF^{x,r}}}(u^\delta-v) \leq \max_{\partial \mF^{x,r}}(u^\delta-v) .$$

    In particular, this implies that $\floorF{u^\delta-v} \leq 0$. But using \eqref{sur-add}, we deduce that
    $$\floorF{u-v}\leq\floorF{u^\delta-v}-\floorF{u^\delta-u}\leq-\eta(\delta)<0\;,$$
    which yields a contradiction. The conclusion is that $M$ cannot be positive, hence $u_\alpha\leq
    v$ in $\mF$ for any $\alpha$ and we get the \GCR by sending $\alpha\to0$.
\end{proof}

Now an important key question is: how can we check \LOCa and \LOCb ? We provide some typical examples.

\

\noindent{\bf The Lipschitz case ---}
We assume that there exists a constant $c>0$ such that the function $\G$ 
satisfies, for all $x \in F$, $z_1 \leq z_2$ and $p,q \in \R^N$
\begin{equation}\label{assump:G-r}
\G(x,z_1,p)-\G(x,z_2,p) \geq c^{-1}(z_1-z_2)\; ,
\end{equation}
\begin{equation}\label{assump:G-p}
|\G(x,z_1,p)-\G(x,z_1,q) \leq c|p-q|\; .
\end{equation}
In the case when \USCS, \LSCS are sets of bounded sub or supersolutions then
\LOCa is satisfied with $u_\alpha(x)=u(x)-\alpha [(|x|^2+1)^{1/2} + c^2]$, indeed
\begin{align*}
 \G(x, u_\alpha(x),Du_\alpha(x)) &\leq \G(x,u(x),Du(x))
    -c^{-1}\alpha [(|x|^2+1)^{1/2} + c^2] +c \alpha \frac{|x|}{(|x|^2+1)^{1/2}}\; ,\\
 &\leq -c^{-1}(\alpha c^2)) +c \alpha =0.
 \end{align*}
Concerning \LOCb, for any $r>0$ we can use
$$u^\delta (y)=u(y)-\delta (|y-x|^2 +k)$$
for some well-chosen constant $k$. Indeed 
\begin{align*}
 \G(y, u^\delta(y),Du^\delta(y)) &\leq 
    \G(y,u(y),Du(y))-c^{-1}\delta (|y-x|^2 +k)+2c\delta|y-x|  \; ,\\
 &\leq -\frac{\delta}{c} ( |y-x|^2 +k-2c^2|y-x|),
 \end{align*}
and with the choice $k=c^4$ we get a subsolution since $X^2-2c^2X+c^4$ has no real roots.
On the other hand, if $y \in \partial \mF^{x,r}$
$$ u^\delta(x)-u(x)=-\delta k \geq -\delta (|y-x|^2 +k)+\delta r^2=u^\delta(y)-u(y)+\delta r^2\; ,$$
so that $\floorF{u^\delta -u}\geq\eta(\delta)=\delta r^2$.

We point out that, even if the assumption on $\G$ are slightly different from the ones we use in the
second particular case we describe above, we could have used similar arguments to treat it.

\

\noindent{\bf The convex case ---}
Here we assume that $\G(x,z,p)$ is convex in $z$ and $p$ and satisfies property~\eqref{assump:G-r}.

For the localization \LOCa, we do not propose any
explicit building of $u_\alpha$ since it strongly depends on (typically) the growth at infinity of
the solutions we want to handle. But a classical construction is described by the following
assumption which emphasizes not only the role of the growth of solutions (via $\psi_1$) but also of
the convexity of $\G$, via the way the $u_\alpha$ are built:
 
\begin{assumption}{\textbf{(Subsol1)}}{Subsolution hypothesis one.}
    For any $u \in \USCS(\mF)$, $v \in \LSCS(\mF)$, there exists an \usc subsolution 
    $\psi_1: F \to \R $ such that for any $0<\alpha < 1$, 
    $u_\alpha (x):=(1-\alpha)u(x) + \alpha \psi_1(x)$ satisfies \LOCa.
\end{assumption}

Concerning \LOCb, the main remark is that, in general, the assumption on the uniform continuity
of $\G$ in $(r,p_x)$ is not satisfied anymore and the above argument based on a perturbation by a term of the
form $-\delta |x-\xb|^2$ does not work. But we may also use a similar construction as for \LOCa
relying on the convexity, assuming for instance

\begin{assumption}{\textbf{(Subsol2)}}{Subsolution hypothesis two.}
    For any $u \in \USCS(\mF)$ and $x \in \mF$, there exists $r>0$ and $\psi_2 \in \USCS(\mF^{x,r})$ such that
    for any $0<\delta < 1$, $u_\delta (y)=(1-\delta)u(y) + \delta \psi_2(y)$ satisfies \LOCb.
\end{assumption}

\noindent 
A typical candidate is $\psi^K_2(x)= -(K|y-x|^2 +k)$ for $k>0$ 
large enough depending on $K$; indeed, thanks to  \eqref{assump:G-r}, $\psi^K_2$ is in
$\USCS(\mF^{x,r})$ if $k$ is sufficiently large. 

It follows that if $y \in \partial \mF^{x,r}$,
$$ u_\delta (y)- u(y) =\delta (\psi^K_2(y)-u(y))\leq -\delta (K r^2-k-u(y))\; ,$$
while $u_\delta (x)- u(x) =-\delta(\tilde k+u(x))$.
Hence, if $|u(z)|\leq m_r$ if $z \in \mF^{x,r}$, we get
$$ \floorF{u_\delta- u}\geq \delta(Kr^2-2m_r)=\eta(\delta)>0\;,$$
if $K$ is chosen large enough. This implies that \LOCb holds.

\subsection{The evolution case}\label{htc:ev}

There are some key differences in the evolution case due to the fact that the time-variable is
playing a particular role since we are mainly solving a Cauchy problem, hence we have to reformulate the
results with the ``parabolic boundary''.
Using here the variable
$X=(x,t)$, we first write the equation as
\begin{equation}\label{GenHJ-evol}
\G(x,t,u,(D_x u,u_t) )= 0\quad \hbox{on  }\mF \times (0,\Tf]\; ,
\end{equation}
where $\mF$ is a closed subset of $\R^N$ and $\G$ is a continuous or discontinuous function on $\mF
\times [0,\Tf] \times \R \times \R^{N+1}$. 

This equation has to be complemented by an initial data at time $t=0$ which can be of an usual form,
namely 
\begin{equation}\label{GenHJ-evol-init1}
u(x,0)= \u0 (x) \quad \hbox{on  }\mF \; ,
\end{equation}
where $\u0$ is a given function defined on $\mF$,
or this initial value of $u$ can be obtained by solving an equation of the type
\begin{equation}\label{GenHJ-evol-init2}
\G_{init} (x,0,u(x,0) ,D_x u (x,0)) = 0 \quad \hbox{on  }\mF \; ,
\end{equation}
where $\G_{init}$ is a continuous or discontinuous function on
$\mF \times [0,\Tf] \times \R \times \R^{N}$.

A strong comparison result for either \eqref{GenHJ-evol}-\eqref{GenHJ-evol-init1} or
\eqref{GenHJ-evol}-\eqref{GenHJ-evol-init2} which is denoted below by \GCR-evol can be defined in
an analogous way as \GCR: subsolutions (in a certain class of functions) are below supersolutions
(in the same class of functions),  $\USCS(\mF)$ and $\LSCS(\mF)$ being just replaced by $\USCS(\mF
\times [0,\Tf])$ and $\LSCS(\mF \times [0,\Tf])$; we just point out that the initial data is included in
the equation in this abstract formulation: for example, a subsolution $u$ satisfies either $$
u(x,0)\leq (\u0)^* (x) \quad \hbox{on  }\mF \; ,$$
in the case of \eqref{GenHJ-evol-init1} or the function $x\mapsto u(x,0)$ satisfies
$$\G_{init} (x,0,u(x,0) ,D_x u (x,0)) \leq 0 \quad \hbox{on  }\mF \; ,
$$
in the viscosity sense, in the case of \eqref{GenHJ-evol-init2}. 

As it is even more clear in the case of \eqref{GenHJ-evol-init2}, a comparison result in the
evolution case consists in two steps 
\begin{enumerate}
\item[$(i)$] proving that for any $u \in \USCS(\mF \times [0,\Tf])$ and $v \in \LSCS(\mF \times
    [0,\Tf])$, 
        \begin{equation}\label{init-ineq-SCR}
        u(x,0) \leq v(x,0) \quad\hbox{on  }\mF,
        \end{equation} 
\item[$(ii)$] showing that this inequality remains true for $t>0$, i.e. 
    $$u(x,t) \leq v(x,t)\quad\hbox{on }\mF\times [0,\Tf]\;.$$
\end{enumerate}
\noindent Of course, in the case of \eqref{GenHJ-evol-init1}, \eqref{init-ineq-SCR} is obvious if
$\u0$ is a continuous function; but, in the case of \eqref{GenHJ-evol-init2}, the proof of such
inequality is nothing but a stationary \GCR in $\mF\times\{0\}$.

Therefore the main additional difficult consists in showing that Property $(ii)$ holds true and we
are going to explain now the analogue of the approach of the previous section {\em assuming that we
have \eqref{init-ineq-SCR}}.

To redefine \LCR, we have to introduce, for $x\in \mF$, $t\in (0,\Tf]$, $r>0$ and $0<h<t$, the sets $$
Q^{x,t}_{r,h}[\mF]:=(B(x,r)\cap \mF)\times (t-h,t ]  \;. $$ 
\label{not:cyl} As in the stationary case, we introduce the set $\USCS(Q^{x,t}_{r,h}[\mF])$, $\LSCS(Q^{x,t}_{r,h}[\mF])$ of respectively \usc subsolutions and 
\lsc supersolution of $G(x,t,u,(D_x u,u_t) )= 0$ in $Q^{x,t}_{r,h}[\mF]$. This means that the viscosity inequalities holds in $Q^{x,t}_{r,h}[\mF]$ 
and not necessarily on its closure, but these sub and supersolutions are \usc or \lsc on $\overline{Q^{x,t}_{r,h}[\mF]}$. 

On the other hand, including $(B(x,r)\cap \mF)\times \{t \}$ in the set where the subsolution or supersolution inequalities hold 
is important in order to have the suitable comparison up to time $t$ and we also refer to Proposition~\ref{jusquaT} 
for the connection between sub and supersolutions in $(B(x,r)\cap \mF)\times (t-h,t )$ and on $(B(x,r)\cap \mF)\times (t-h,t ]$.

With this definition we have

\index{Comparison result!local (evolution)}\label{page:LCRevol}
\begin{assumption}{\LCREV}{Local comparison result -- evolution case.}
    For any $(x,t) \in \mF \times (0,\Tf]$, there exists $\bar r>0$, $0<\bar h<t$ such that, for any $0<r\leq \bar r$, $0<h<\bar h$, 
    if $u \in \USCS(Q^{x,t}_{\bar r,\bar h}[\mF])$, $v \in \LSCS(Q^{x,t}_{\bar r,\bar h}[\mF])$,
    $$ \max_{\overline{Q^{x,t}_{r,h}[\mF]}}(u-v)_+ \leq \max_{\partial_p Q^{x,t}_{r,h}[\mF]}(u-v)_+\;,$$
    where $\partial_p Q^{x,t}_{r,h}[\mF]$ stands for the parabolic boundary of $Q^{x,t}_{r,h}[\mF]$, 
    composed of a ``lateral'' part and an ``initial'' part as follows
    $$\begin{aligned}\partial_p Q^{x,t}_{r,h}[\mF] &= \Big\{(\partial {B(x,r)}\cap \mF)\times
    [t-h,t]\Big\}\bigcup \Big\{(\overline{B(x,r)}\cap \mF)\times \{t-h\}\Big\}\\
    &=: \partial_\mathrm{\,lat} Q \cup \partial_\mathrm{\,ini}Q\;.\end{aligned}
    $$
\end{assumption}

We point out that, in the sequel, we are going to play with the parameters $r,h$ to obtain the comparison result. This explains the formulation 
of \LCR where the local comparison result has to hold in $Q^{x,t}_{r,h}$ for any $0<r\leq \bar r$, $0<h\leq \bar h$.
 
The corresponding evolution versions of \LOCa and \LOCb are given by

\label{page:LOCevol1}
\begin{assumption}{\LOCaEV}{Localization assumption one -- evolution case.}
    If $\mF$ is unbounded, for any $u \in \USCS(\mF\times [0,\Tf])$, for any $v \in \LSCS(\mF\times
    [0,\Tf])$, there exists a sequence $(u_\alpha)_{\alpha>0}$ of \usc subsolutions of \eqref{GenHJ-evol}
    such that $u_\alpha(x,t) - v (x,t) \to -\infty$ when $|x|\to +\infty$, $x\in \mF$, uniformly for
    $t\in [0,\Tf]$. Moreover, for any $x\in \mF$, $u_\alpha(x,t) \to u(x,t)$ when $\alpha \to 0$.
\end{assumption}

\label{page:LOCevol2}
\begin{assumption}{\LOCbEV}{Localization assumption two -- evolution case.}
    For any $x \in \mF$, if $u \in \USCS(Q^{x,t}_{\bar r,\bar h}[\mF])$ for some $0<\bar r$, $0<\bar h <t$, there
    exists $0<h\leq \bar h$ and a sequence $(u^\delta)_{\delta>0}$ of functions in
    $\USCS(Q^{x,t}_{\bar r,h}[\mF])$ such that $\floorFt{u^\delta-u}\geq\tilde\eta(\delta)>0$ with
    $\tilde\eta(\delta)\to0$ as $\delta\to0$.  Moreover $u^\delta \to u$ uniformly on $\overline{Q^{x,t}_{r,h}[\mF]}$
    when $\delta \to 0$.
\end{assumption}

Notice that \LOCbEV is only concerned with a property at the lateral boundary. As we see in the
proof, the initial boundary is easily left out by a minimality argument.

With these assumptions, we have the
\begin{proposition}\label{reducCR-evol}\emph{--- Reduction to a Local Comparison Result, evolution
    case.}\smsp
Assuming \LOCaEV and \LOCbEV,  \LCREV implies \GCREV.
\end{proposition}

\begin{proof}There is no main change in the proof except the following point: using \LOCaEV, we
    may assume that the maximum of $u-v$ is achieved at some point $(x,t)$. Here we choose $t$ as
    the minimal time such that we have a maximum of $u-v$. And we assume that this maximum is
    positive.

    \smallskip

    \noindent\textbf{(a)}
    Notice first that $t>0$ because $u\leq v$ on $\mF\times \{0\}$ and, if $r$ and $h\leq \bar h$ are given by \LCREV, notice also that 
    by the minimality property of $t$,
    $$ \max_{(\overline{B(x,r)}\cap\mF)\times\{t-h\}}(u-v)<
    \max_{\overline{Q^{x,t}_{r,h}[\mF]}}(u-v)= u(x,t)-v(x,t)\;.$$
    In other words, the maximum of $u-v$ is not attained
    on the initial boundary, $\partial_\mathrm{\,ini}Q$.
    On the other hand, on the lateral boundary we obviously get
    $$\floorFt{u-v}=(u-v)(x,t)-\max_{\partial_\mathrm{\,lat} Q^{x,t}_{r,h}[\mF]}(u-v)\geq0\;.$$

    \noindent\textbf{(b)}
    Then we apply \LOCbEV. Using the properties of the sequence $(u^\delta)_{\delta>0}$, we can choose $\delta$ 
    small enough in order that again, the maximum of $u^\delta-v$ is not attained at time  $t-h$.
\begin{align}
    u^\delta (x,t)-v(x,t) & \leq \max_{\overline{Q^{x,t}_{r,h}[\mF]}}(u^\delta-v) \\
    &\leq \max_{\partial_p Q^{x,t}_{r,h}[\mF]}(u^\delta-v)= \max_{\partial_\mathrm{\,lat} Q^{x,t}_{r,h}[\mF]}(u^\delta-v)\;.
\end{align}
    In other words, $\floorFt{u^\delta-v}\leq0$ and the rest of the proof follows the same arguments
    as in the stationary case 
$$\floorFt{u-v}\leq
    \floorFt{u^\delta-v}-\floorFt{u^\delta-u}\leq-\tilde\eta(\delta)<0\;,$$
    which leads to a contradiction.
\end{proof}

Now we consider \eqref{GenHJ-evol} and the assumptions on $\G$ for the Lipschitz case are:
there exists a constant $c>0$ such that, for all $x \in \F$, $t\in [0,\Tf]$, $z_1 \leq z_2$, $p_t^1\leq p_t^2$, $p_x^1,p_x^2 \in \R^N$
\begin{align}
    & \G(x,t,z_2,(p_x^2,p_t^2))-\G(x,t,z_1,(p_x^1,p_t^1)) \geq c^{-1}(p_t^2-p_t^1)\; ,
    \label{assump:G-pt} \\[2mm]
    &|\G(x,t,z_1,(p_x^2,p_t^1))-\G(x,t,z_1,(p_x^1,p_t^1)) \leq c|p_x^2-p_x^1|\; .
    \label{assump:G-px}
\end{align}

In particular, Assumption~\eqref{assump:G-pt} is a key property and, building the $u_\alpha$ and $u^\delta$ turns
out to be easy. Indeed
$$u_\alpha(x,t)=u(x,t)-\alpha [(|x|^2+1)^{1/2} +Kt]\; ,$$
for $K>0$ large enough. And for $u^\delta$,
$$u^\delta(y,s)=u(y,s)-\delta [(|y-x|^2+1)^{1/2}-1 +K(s-t)]\; ,$$
where $K$ has to be chosen large enough to have a subsolution and $h$ small enough to have the right property on the parabolic boundary. 
This is because of this property on the parabolic boundary that \LOCbEV has this formulation for $h$.

In the convex case, Assumption~\eqref{assump:G-pt} still holds but Assumption~\eqref{assump:G-px} is replaced by
the fact that $(p_x,p_t)\mapsto \G(x,t,z,(p_x,p_t))$ is convex for any $x\in \F$, $t\in [0,\Tf]$, $z\in \R$ and by the fact that $\G(x,t,0,(0,0))$ is bounded from above. Then, we build $u_\alpha$ and $u^\delta$
in the following way
$$ u_\alpha(x,t)=(1-\alpha) u(x,t)+\alpha \chi(x,t)\; ,$$
where $\chi(x,t):=[(|x|^2+1)^{1/2} +Kt]$. For $K>0$ large enough, the above assumptions imply that $\chi$ is a subsolution of the $\G$-equation and so is $u_\alpha$ by convexity. We may even take $K$ larger in order that $\chi$ and $u_\alpha$ are stict subsolutions.

On the other hand, for $u^\delta$,
$$u^\delta(y,s)=(1-\delta) u(y,s)+\delta \psi^K(y,s)\; ,$$
where $\psi^K(y,s):= -K[(|y-x|^2+1)^{1/2}-1] -k(s-t)\; .$ Again for any $K>0$, there exists $k>0$ large enough such that $\psi^K$ is a subsolution and so is $u^\delta$ by convexity. Moreover it is clear that $u^\delta \to u$ uniformly on $\overline{Q^{x,t}_{r,h}[\mF]}$.

It remains to evaluate $\floorFt{u^\delta-u}\geq\tilde\eta(\delta)>0$. If $(y,s) \in \partial_{lat} Q$ then
$$ (u^\delta-u)(y,s)= \delta[\psi^K(y,s)-u(y,s)]\leq\delta\left(K[(r^2+1)^{1/2}-1] + kh-u(y,s)\right)\; ,$$
while $(u^\delta-u)(x,t)=-\delta u(x,t)$. Hence
$$\floorFt{u^\delta-u}\geq \delta\left(K[(r^2+1)^{1/2}-1] - kh+u(y,s)-u(x,t)\right)\; .$$
If $\displaystyle m_r=\max_{\overline{Q^{x,t}_{r,h}[\mF]}} |u(y,s)|$, we have
$$\floorFt{u^\delta-u}\geq \delta\left(K[(r^2+1)^{1/2}-1] - kh+2m_r\right)\; .$$
The new point here is that we have to choose $h$ small enough in order that $kh \leq 2^{-1}K[(r^2+1)^{1/2}-1]$, which
gives
$$\floorFt{u^\delta-u}\geq \delta\left(2^{-1} K[(r^2+1)^{1/2}-1] +2m_r\right)\; ,$$
and the choice of $K$ large enough provides the desired property.

\begin{remark}\label{rem:htc}We are going to use these localization properties throughout the book in order
    to treat discontinuities, so let us make two important comments here.
    \begin{enumerate}
        \item[$(i)$] As the proofs show, both in the stationary and evolution case, in order to
            have \GCR, we do not need \LCR to hold on the whole set $\mF$: indeed, if we already
            know that $u\leq v$ on some subset $\mA$ of $\mF$, then \LCR is required only in
            $\mF\setminus \mA$.
        \item[$(ii)$] Both in the Lipschitz and convex case we can check \LOCa, \LOCb---and their
            evolution variations---in standard ways. It should be noticed that, in both cases, the
            localization procedure is independent of the possible discontinuities in the
            $x$-variable. Which is why it will be systematically applied to get the various \GCR
            throughout this book as a first step.
          \item[$(iii)$] The above checking of \LOCa, \LOCb---and their evolution
              variations---strongly relies on either \eqref{assump:G-r} or \eqref{assump:G-pt} and
              does not allow to take into account important examples involving gradient constraints, for
              instance: 
              $$\max(\G(x,u,D_xu);|D_x u|-1)=0\;.$$
              Indeed, the quadratic perturbation above is not be compatible with the constraint in general.
              However, we point out that such situations can be handled under suitable assumptions;
              the reader may have a look at Lemma~\ref{lem:loc-strat} in the proof of the
              comparison result in the stratified setting where we develop this idea.
    \end{enumerate}
\end{remark}

\subsection{Viscosity inequalities at $t=\Tf$ in the evolution case} 

We conclude this section by examining the viscosity sub and supersolutions inequalities at $t=\Tf$ and
their consequences on the properties of sub and supersolutions. To do so, we have to be a little bit
more precise on the assumptions on the function $\G$ appearing in \eqref{GenHJ-evol}. We introduce
the following hypothesis

\label{page:HGCP} 
\begin{assumption}{\HGCP}{Basic Assumption for the evolution case.}
\emph{For any $(x,t,r,p_x,p_t)\in \mF\times(0,\Tf]\times\R\times\R^N\times\R$,
the function $p_t\mapsto \G\big(x,t,r,(p_x ,p_t)\big)$ is increasing and $\G\big(x,t,r,(p_x ,p_t) \big)
\to +\infty$ as $p_t \to +\infty$, uniformly for bounded $x,t,r,p_x$.}
\end{assumption}

This assumption is obviously satisfied in the standard case, $i.e.$ for equations like
$$u_t + H(x,t,u,D_x u)=0\quad \hbox{in  }\R^N \times (0,\Tf]\; ,$$
provided $H$ is continuous (or only locally bounded) since in this case 
$\G(x,t,r,(p_x ,p_t) )=p_t+H(x,t,r,p_x)$.

\begin{proposition}\label{jusquaT}Under assumption \HGCP, we have
\begin{enumerate}
\item[$(i)$] If $u: \mF \times (0,\Tf) \to \R$ \resp{$v: \mF \times (0,\Tf) \to \R$} is an \usc viscosity subsolution \resp{lsc supersolution} of 
$$ \G(x,t,w,(D_x w,w_t) )= 0\quad \hbox{on  }\mF \times (0,\Tf)\; ,$$
then, for any $0<T'<\Tf$, $u$ \resp{$v$} is an \usc viscosity subsolution \resp{lsc supersolution} of 
$$ \G(x,t,w,(D_x w,w_t) )= 0\quad \hbox{on  }\mF \times (0,T']\; .$$
\item[$(ii)$] Under the same conditions on $u$ and $v$ and if 
\begin{equation}\label{prop:T:sub.semic}
u(x,\Tf) = \limsup_{\substack{ (y,s)\to (x,\Tf)\\ s<\Tf}} u(y,s) \quad [\hbox{resp.} \quad v(x,\Tf) = \liminf_{\substack{ (y,s)\to (x,\Tf)\\ s<\Tf}} v(y,s)]\; ,
\end{equation}
then $u$ and $v$ are respectively sub and supersolution of \eqref{GenHJ-evol}.

\item[$(iii)$] If $u: \mF \times (0,\Tf] \to \R$ is an \usc viscosity subsolution of \eqref{GenHJ-evol},
 then, for any $x\in \mF$, \eqref{prop:T:sub.semic} holds for $u$.

\item[$(iv)$] If $\G$ satisfies $\G(x,t,r,(p_x ,p_t) ) \to -\infty$ as $p_t \to -\infty$, uniformly for bounded $x,t,r,p_x$ and if $v: \mF \times (0,\Tf] \to \R$ is a \lsc viscosity supersolution of \eqref{GenHJ-evol}, then \eqref{prop:T:sub.semic} holds for $v$.
\end{enumerate}
\end{proposition}

This result clearly shows the particularities of the viscosity inequalities at the terminal time $t=\Tf$ or $t=T'$: 
sub and supersolutions in $\mF \times (0,\Tf)$ are automatically sub and supersolutions on $\mF \times (0,T']$ 
for any $0<T'<\Tf$ and even for $T'=\Tf$ provided that they are extended in the right way up to time $\Tf$, 
according to \eqref{prop:T:sub.semic}. And conversely sub and supersolutions on $\mF \times (0,\Tf]$ 
satisfy \eqref{prop:T:sub.semic} provided that $G$ has some suitable properties which clearly hold 
for the standard $H$-equation above. Here there is a difference between sub and supersolutions due 
to the disymmetry of Assumption~\HGCP. We will come back later on this point with the control interpretation.

\begin{proof}
    We only prove the first and second part of the result in the subsolution case, the proof for the
    supersolution being analogous.  

    \

    \noindent\textbf{(a)}
    Let $\varphi$ be a smooth function (say, in $\mF \times [0,\Tf]$) and let $(x,T')$ be a strict
    local maximum point of $u-\varphi$ in $\mF \times [0,T']$. We introduce the function
    $$ (y,s)\mapsto u(y,s)-\varphi(y,s) - \frac{[(s-T')_+]^2}{\eps}\; .$$
    An easy application of Lemma~\ref{lem:cv-pen} implies that this function has a local maximum
    point at $(\xe,\te)$ and we have 
    $$ (\xe,\te)\to (x,T')\quad\hbox{and}\quad u(\xe,\te)\to u(x,T')\; \hbox{as  }\e \to 0\; ,$$ 
    because of both the strict maximum point property and the $\eps$-penalisation. Moreover, for
    $\eps$ small enough, the penalization implies that $t_\e<\Tf$.

    Since $u$ is a subsolution of the $\G$-equation in $\mF \times (0,\Tf)$ and as we noticed,
    $(\xe,\te)$ is a local maximum point in $\mF \times (0,\Tf)$, we have 
    $$ \G_*\Big(\xe,\te, u(\xe,\te), (D_x\varphi(\xe,\te), \varphi_t (\xe,\te) 
    + 2\e^{-1}(s-\Tf)_+)\Big)\leq 0\; .$$
    But, by \HGCP , $\G(y,s,r,(p_x,p_t))$ and therefore $\G_*(y,s,r,(p_x,p_t))$ is increasing in the
    $p_t$-variable and we have 
    $$ \G_*(\xe,\te, u(\xe,\te), (D_x\varphi(\xe,\te), D_t \varphi(\xe,\te)))\leq 0\; .$$
    The conclusion follows from the lower semicontinuity of $\G_*$ by letting $\e$ tend to $0$.

    \

    \noindent\textbf{(b)}
    For the proof of $(ii)$, we argue in an analogous way: if $(x,\Tf)$ is a strict local maximum
    point of $u-\varphi$ in $\mF \times [0,\Tf] $, we introduce the function 
    $$ (y,s)\mapsto u(y,s)-\varphi(y,s) - \frac{\e}{(\Tf-s)}\; .$$
    By Lemma~\ref{lem:cv-pen}, this function has a local maximum point at $(\xe,\te)$ and we have
    $$ (\xe,\te)\to (x,\Tf)\quad\hbox{and}\quad u(\xe,\te)\to u(x,\Tf)\; \hbox{as  }\e \to 0\; .$$ 
    It is worth pointing out that, in this case, the proof of such properties uses not only the
    strict maximum point property and the fact that the $\eps$-penalisation is vanishing, but also
    strongly Property~\eqref{prop:T:sub.semic} for $u$ which provides Assumption-(iii) of
    Lemma~\ref{lem:cv-pen}.

    We are led to
    $$ \G_*\Big(\xe,\te, u(\xe,\te), (D_x\varphi(\xe,\te), \varphi_t (\xe,\te) 
    + \frac{\e}{(\Tf-s)^2})\Big)\leq 0\; ,$$
    and we conclude by similar arguments as in the proof of $(i)$.

    \

    \noindent\textbf{(c)}
    Finally we prove $(iii)$ since the supersolution one, $(iv)$, follows again from similar arguments
    with the additional assumption on $\G$.

    We pick some $(x,\Tf)\in \mF \times \{\Tf\}$ and we aim at proving \eqref{prop:T:sub.semic}. We
    argue by contradiction: if this is not the case then $u(x,\Tf)>\limsup u(y,s)$ as
    $(y,s)\to(x,\Tf)$, with $s<\Tf$. This implies that for any $\e>0$ small enough and any $C>0$,
    the function
    $$ (y,s)\mapsto u(y,s) - \frac{|y-x|^2}{\e^2} -C(s-\Tf)$$ 
    can only have a maximum point for $s=\Tf$, say at $y=\xe$ close to $x$. The viscosity
    subsolution inequality reads 
    $$ \G_*\Big(\xe,\Tf, u(\xe,\Tf), (\frac{2(\xe-x)}{\e^2}, C)\Big)\leq 0\; .$$
    But if we fix $\e$ (small enough), all the arguments in $\G_*$ remains bouded, except $C$.  So,
    choosing $C$ large enough, we have a contradiction because of \HGCP.
\end{proof}

\begin{remark}\label{rem:MaT}
    We point out that, even if Proposition~\ref{jusquaT} only provides the result for sub or
    supersolutions inequalities in sets of the form $\mF\times (0,\Tf)$, a similar result can be
    obtained, under suitable assumptions, for sub and supersolution properties at any point
    $(x,\Tf)$ of $\mathcal{M}$ where $\mathcal{M}$ is the restriction to $\R^N\times(0,\Tf]$ to a
    submanifold of $\R^N\times \R$. Indeed, it is clear from the proof that only Assumption \HGCP is
    really needed to have such properties.  
\end{remark}

\subsection{The simplest examples of comparison results: the continuous case}
\label{simple-ex-comp}

\index{Comparison result!simplest examples}As a simple example, we consider the standard continuous Hamilton-Jacobi Equation
\begin{equation}\label{HJ:evol}
u_t + H(x,t,u,D_x u)=0 \quad \hbox{in  }\R^N\times (0,\Tf)\, ,
\end{equation}
where $H: \R^N\times [0,\Tf]\times \R \times \R^N \to \R$ is a continuous function, $u_t$ denotes
the time-derivative of $u$ and $D_x u$ is the derivative with respect to the space variables $x$. Of
course, this equation has to be complemented by an initial data 
\begin{equation}\label{id:HJ:evol}
    u(x,0)=u_0 (x) \quad \hbox{in  }\R^N\; .
\end{equation}
In this section, we always assume that $u_0\in C (\R^N)$.

We provide comparison results in the two cases we already consider above, namely the Lipschitz case and
the convex case, the later one allowing more general Hamiltonians coming from unbounded control problems.
In order to formulate the results, let us introduce
\begin{enumerate}
    \item[$(i)$] $\USCS(\R^N\times
        [0,\Tf])$ the set of \usc subsolution $u$ of \eqref{HJ:evol} such that $u(x,0)\leq
        u_0(x)$ in $\R^N$;
    \item[$(ii)$] $\LSCS(\R^N\times [0,\Tf])$ is the set of  \lsc supersolutions $v$
         of \eqref{HJ:evol} such that $v(x,0)\geq u_0(x)$ in $\R^N$. 
\end{enumerate}
Our result is the following
\begin{theorem}\label{comp:LC}\emph{--- Comparison for the Lipschitz case}\smsp 
    Under assumption \HBAHJ, a \GCREV holds for bounded sub and supersolutions of
    \eqref{HJ:evol}-\eqref{id:HJ:evol} in $\USCS(\R^N\times [0,\Tf])$ and $\LSCS(\R^N\times
    [0,\Tf])$ respectively.  
\end{theorem}

\begin{proof}
    We just sketch it since it is the standard comparison proof that we recast in a little unsual way. 

    \

    \noindent\textbf{(a)}
    By the arguments of the previous section, it suffices to prove \LCREV. Therefore, we argue in
    $\overline{Q^{\xb,\tb}_{r,h}}$ for some $\xb \in \R^N$, $0<\tb<\Tf$, $r,h>0$ and we assume that
    $\displaystyle \max_{\overline{Q^{\xb,\tb}_{r,h}}}(u-v) >0$  where $u \in
    \USCS(Q^{\xb,\tb}_{r,h})$, $v \in \LSCS(Q^{\xb,\tb}_{r,h})$. 

    It is worth pointing out that, in $\overline{Q^{\xb,\tb}_{r,h}}$, taking into account the fact
    that $u$ and $v$ are bounded, we have fixed constants and modulus in \hyp{BA-HJ} that we
    denote below by $C_1,\gamma$ and $m$. Moreover, we can assume \wlg that $\gamma>0$ through
    the classical change $u(x,t)\to \exp(Kt)u(x,t)$, $v(x,t)\to \exp(Kt)v(x,t)$ for some large
    enough constant $K$.
    
    \

    \noindent\textbf{(b)}
    We argue by contradiction, assuming that
    $$ \max_{\overline{Q^{\xb,\tb}_{r,h}}}(u-v) > \max_{\partial_p Q^{\xb,\tb}_{r,h}}(u-v)\;,$$
    and we introduce the classical doubling of variables
    $$ (x,t,y,s) \mapsto u(x,t)-v(y,s)-\frac{|x-y|^2}{\e^2}-\frac{|t-s|^2}{\e^2}\; .$$
    Using Lemma~\ref{lem:cv-pen}, this \usc function has a maximum point at $(\xe,\te,\ye,\se)$ with
    $(\xe,\te),(\ye,\se) \in Q^{\xb,\tb}_{r,h}$ and 
    $$ u(\xe,\te)-v(\ye,\se)\to \max_{\overline{Q^{\xb,\tb}_{r,h}}}(u-v)\quad \hbox{and}\quad
    \frac{|\xe-\ye|^2}{\e^2}+\frac{|\te-\se|^2}{\e^2}\to 0\; .$$
    It remains to write the viscosity inequalities which reads
    $$ a_\e + H(\xe,\te,u(\xe,\te),p_\e)\leq 0\quad \hbox{and}\quad  a_\e + H(\ye,\se,v(\ye,\se),p_\e)\geq 0\; ,$$
    with
    $$ a_\e= \frac{2(\te-\se)}{\e^2}\quad \hbox{and}\quad p_\e =\frac{2(\xe-\ye)}{\e^2}\; .$$
    Subtracting the two inequalities, we obtain
    $$ H(\xe,\te,u(\xe,\te),p_\e)-H(\ye,\se,v(\ye,\se),p_\e)\leq 0\; ,$$
    that we can write as
    $$ \begin{aligned}
        \Big[ H(\xe,\te,u(\xe,\te),p_\e)- & H(\xe,\te,v(\xe,\te),p_\e)\Big ]\\
        \leq & \Big[ H(\xe,\te,v(\xe,\te),p_\e)-H(\ye,\se,v(\ye,\se),p_\e) \Big]\;.
    \end{aligned}
    $$

    \noindent\textbf{(c)}
    It remains to apply \hyp{BA-HJ}, leading to
    $$ \gamma(u(\xe,\te),p_\e)-v(\xe,\te)) -C_1(|\xe-\ye|+|\te-\se|)|p_\e|-m(|\xe-\ye|+|\te-\se|)\leq 0\; .$$
    But, as $\e\to 0$, $m(|\xe-\ye|+|\te-\se|)\to 0$ since $|\xe-\ye|+|\te-\se|=o(\e)$ and
    $$(|\xe-\ye|+|\te-\se|)|p_\e|= \frac{2|\xe-\ye|^2}{\e^2}+ \frac{2|\te-\se||\xe-\ye|}{\e^2} \to 0\; .$$
    Therefore we have a contradiction for $\e$ small enough since
    $$\gamma(u(\xe,\te),p_\e)-v(\xe,\te))\to \gamma \max_{\overline{Q^{\xb,\tb}_{r,h}}}(u-v) >0\; .$$ 
    And the proof is complete.
\end{proof}

It is worth pointing out the simplifying effect of the localization argument in this proof: the core
of the proof becomes far simpler since we do have to handle several penalization terms at the same
time (the ones for the doubling of variables and the localization ones).

We have formulated and proved Theorem~\ref{comp:LC} in a classical way and in a way which is
consistent with the previous sections but in this Lipschitz framework, we may have the stronger
result based on {\em a finite speed of propagation type phenomena} which we present here since it
follows from very similar arguments
\index{Finite speed of propagation}
\begin{theorem}\label{comp:FSP}\emph{--- Finite speed of propagation}\smsp
    Assume that \hyp{BA-HJ} holds with $\gamma(R)$ independent of $R$. Let $u$ be a bounded \usc subsolution
    of \eqref{HJ:evol} and $v$ be a bounded \lsc supersolution of \eqref{HJ:evol}.
    If $u(x,0) \leq v(x,0)$ for $|x|\leq R$ for some $R>0$, then
    $$ u(x,t) \leq v(x,t) \quad \hbox{for  }|x| \leq R-C_2t\;,$$
    where $C_2$ is given by \hyp{BA-HJ}.
\end{theorem}
\begin{proof}
    Let $\chi : (-\infty, R) \to \R$ be a smooth function such that $\chi(s) \equiv 0$ if
    $s\leq 0$, $\chi$ is increasing on $\R$ and $\chi(s)\to +\infty$ when $s \to R^-$. We set $$
    \psi(x,t) := \exp(-|\gamma| t) \chi(|x|+C_2t)\; .$$
    This function is well-defined in $\mathcal{C}:=\{(x,t):\ |x|+C_2t \leq R\}$. 

    We claim that, for $0<\alpha\ll 1$, the function $u_\alpha(x,t):= u(x,t)-\alpha  \psi(x,t)$ in a
    subsolution of \eqref{HJ:evol} in $\mathcal{C}$ and satisfies $u_\alpha(x,t)\to -\infty$ if
    $(x,t)\to \partial \mathcal{C}\cap \{t>0\}$ and $u_\alpha(x,0) \leq u(x,0)$ for $|x|\leq R$. 

    The second part of the claim is obvious by the properties of $\psi$. To prove the first one, we
    first compute formally $$ (u_\alpha)_t + H(x,t,u_\alpha,D_x u_\alpha)\leq u_t + H(x,t,u,D_x u)
    -\alpha (\psi_t -|\gamma| \psi -C_2|D_x \psi|)\; .$$ But an easy---again formal---computation
    shows that $\psi_t -|\gamma| \psi -C_2|D_x \psi| \geq 0$ in $\mathcal{C}$ and since the
    justification of these formal computations is straightforward by regularizing $|x|$ in order
    that $\psi$ becomes $C^1$, the claim is proved.

    The rest of the proof consists in comparing $u_\alpha$ and $v$ in $\mathcal{C}$, which follows from
    the same arguments as in the proof of Theorem~~\ref{comp:LC}.
\end{proof}

Now we turn to the {\em convex case} where we may have some more general behavior for $H$ and in
particular no Lipschitz continuity in $p$. To simplify the exposure, we do not formulate the
assumption in full generality but in the most readable way:

\label{page:HJConv} 
\begin{assumption}{\HBAConv}{Basic assumptions in the convex case.}
    $H(x,t,r,p)$ is a locally Lipschitz function which is convex in $(r,p)$. Moreover, for any ball
    $B\subset \R^N\times [0,\Tf]$, for any $R>0$, there exists constants $L=L(B,R), K=K (B,R) >0$
    and a function $G=G(B,R):\R^N\to [1,+\infty[$ such that, for any $x,y \in B$, $t,s \in [0,\Tf]$,
    $-R \leq u \leq v \leq R$ and $p \in \R^N$ $$ D_p H(x,t,r,p)\cdot p -H(x,t,u,p) \geq G(p) -L\;
    ,$$ $$ |D_x H(x,t,r,p)|, |D_t H(x,t,r,p)| \leq K G(p)( 1+ |p|)\; ,$$ $$ D_r H(x,t,r,p) \geq 0 \;.$$
\end{assumption} \\[-1.2cm]

On the other hand, we assume the existence of a subsolution

\begin{assumption}{\HSubHJ}{Assumption on the existence of a subsolution.}
    \label{page:HJSub} There exists an $C^1$-function $\psi : \R^N \times [0,\Tf] \to \R$ which is a
    subsolution of \eqref{HJ:evol} and which satisfies $\psi(x,t) \to -\infty$ as $|x| \to +\infty$,
    uniformly for $t\in [0,\Tf]$ and $\psi (x,0) \leq u_0(x) $ in $\R^N$.
\end{assumption}

Let us now introduce the sets
\begin{enumerate}
    \item[$(i)$] $\USCS^\psi(\R^N\times [0,\Tf])$, of bounded \usc subsolution $u$ of
        \eqref{HJ:evol} satisfying 
        $$ \limsup_{|x|\to +\infty} \frac{u(x,t)}{\psi(x,t)} \geq0\quad
        \text{uniformly for $t\in [0,\Tf]$}\;.$$
    \item[$(ii)$] $\LSCS^\psi(\R^N\times [0,\Tf])$, of bounded \lsc supersolutions $v$ 
        of \eqref{HJ:evol} satisfying
        $$\liminf_{|x|\to +\infty} \frac{v(x,t)}{\psi(x,t)} \leq 0\quad 
        \text{uniformly for $t\in [0,\Tf]$}\;.$$
\end{enumerate}

The result is the
\begin{theorem}\label{comp:CC}\emph{--- Comparison in the Convex case.}\smsp
    Assume \hyp{BA-HJ-U} and \HSubHJ. Then a \GCREV holds for sub and supersolutions of
    \eqref{HJ:evol}-\eqref{id:HJ:evol} in $\USCS^\psi(\R^N\times [0,\Tf])$ and
    $\LSCS^\psi(\R^N\times [0,\Tf])$ respectively.
\end{theorem}

\begin{proof} We use a similar approach as in the Lipschitz case, with a few modifications.

    \smallskip

    \noindent\textbf{(a)}
    The first step consists in replacing $u$ by $u_\alpha := (1-\alpha)u +\alpha \psi$ for
    $0<\alpha \ll 1$. The convexity of $H(x,t,r,p)$ in $(r,p)$ implies that $u_\alpha$ is still a
    subsolution of \eqref{HJ:evol} and $u_\alpha (x,0) \leq u_0(x)$ in $\R^N$. Moreover, by the
    definition of $\USCS(\R^N\times [0,\Tf])$ and $\LSCS(\R^N\times [0,\Tf])$, 
    $$ \lim (u_\alpha (x,t)-v(x,t)) = -\infty \quad\hbox{as  $|x| \to +\infty$, 
    uniformly for $t\in [0,\Tf]$.}$$
    Therefore the subsolution $\psi$ plays its localization role.

    \smallskip

    \noindent\textbf{(b)}
    For \LCREV, we argue exactly in the same way as in the proof of Theorem~\ref{comp:LC} in
    $Q^{\xb,\tb}_{r,h}$---therefore with fixed contants $L,K$ and a fixed function $G$---but with
    the following preliminary reductions: changing $u,v$ in $u(x,t)+Lt$ and $v(x,t)+Lt$, we may
    assume that $L=0$. Finally we perform Kru\v{z}kov's change of variable
    $$ \tilde u(x,t):=-\exp(-u(x,t))\quad , \quad \tilde v(x,t):=-\exp(-v(x,t))\; .$$
    The function $\tilde u,\tilde v$ are respectively sub and supersolution of 
    $$ w_t + \tilde H(x,t,w,Dw)=0\quad\hbox{ in  }Q^{\xb,\tb}_{r,h}\; ,$$
    with $\tilde H(x,t,r,p)= -r H(x,t,-\log(-r), -p/r)$.

    Computing $D_r \tilde H(x,t,r,p)$, we find $(D_p H \cdot p -H)(x,t,-\log(-r), -p/r)) \geq
    G(-p/r)$, while $D_x \tilde H(x,t,r,p)$, $D_t \tilde H(x,t,r,p)$ are estimated by $|r| |D_x
    H(x,t,-\log(-r), -p/r)|$, $|r| |D_t H(x,t,-\log(-r), -p/r)|$, i.e. by $|r| K G(-p/r)( 1+
    |p/r|)$.

    \smallskip

    \noindent\textbf{(c)}
    Following the proof of Theorem~\ref{comp:LC}, we have to examine an inequality like
    $$ \tilde H(\xe,\te,\tilde  u(\xe,\te),p_\e)-\tilde 
    H(\ye,\se,\tilde  v(\ye,\se),p_\e)\leq 0\; .$$
    To do so, we argue as if $\tilde H$ was $C^1$ (the justification is easy by a standard
    approximation argument) and we introduce the function 
    $$ f(\mu):=  \tilde H(\mu \xe+ (1-\mu)\ye,\mu \te+ (1-\mu)\se ,\mu 
    \tilde  u(\xe,\te)++ (1-\mu)\tilde  v(\ye,\se) ,p_\e)\; ,$$
    which is defined on $[0,1]$. The above inequality reads $f(1)-f(0)\leq 0$ while
    $$ f'(\mu) = D_x \tilde H. (\xe-\ye) + D_t \tilde H .(\te-\se) 
    + D_r \tilde H.(\tilde  u(\xe,\te)-\tilde  v(\ye,\se))\; ,$$
    where all the $\tilde H$ derivatives are computed at the point 
    $$(\mu \xe+ (1-\mu)\ye ,\mu \te+ (1-\mu)\se ,\mu \tilde  u(\xe,\te)+ (1-\mu)\tilde  v(\ye,\se) ,p_\e)\; .$$
    If we denote by $r_\e = \mu \tilde  u(\xe,\te)+ (1-\mu)\tilde  v(\ye,\se)$, we have, by the above estimates,
    \begin{align*}
     f'(\mu) \geq & - |r_\e| K G(-p_\e/r_\e)( 1+ |p_\e/r_\e|)(|\xe-\ye| + |\te-\se|) \\
     &+ G(-p_\e/r_\e).(\tilde  u(\xe,\te)-\tilde  v(\ye,\se))\\
     \geq & G(-p_\e/r_\e)\Big[ -  K  (|r_\e|+ |p_\e)(|\xe-\ye| + |\te-\se|) + (\tilde
        u(\xe,\te)-\tilde  v(\ye,\se))\Big] \; .
     \end{align*}
    But if $\displaystyle M:=\max_{\overline{Q^{\xb,\tb}_{r,h}}}(\tilde u-\tilde v) >0$, the
    arguments of the proof of Theorem~\ref{comp:LC} show that the bracket is larger than $M/2$ if $\e$
    is small enough. Therefore $ f'(\mu) \geq M/2>0$, a contradiction with $f(1)-f(0)\leq 0$.
\end{proof}

We conclude this part by an application of Theorem~~\ref{comp:LC} and ~\ref{comp:CC}.

\begin{example}
    We consider the equation
    $$ u_t + a(x,t)|D_x u|^q -b(x,t)\cdot D_x u = f(x,t) \quad \hbox{in  }\R^N\times (0,\Tf)\, ,$$
    where $a,b, f$ are at least continuous function in $\R^N\times [0,\Tf]$ and $q\geq 1$.

    Of course, Theorem~~\ref{comp:LC} applies if $q=1$ and $a,b$ are locally Lipschitz continuous
    functions and $f$ is a uniformly continuous function on $\R^N\times [0,\Tf]$.
    Theorem~~\ref{comp:CC} is concerned with the case $q>1$ and $a(x,t) \geq 0$ in $\R^N\times
    [0,\Tf]$ in order to have a convex Hamiltonian.

    Next the computation gives
    $$ D_p H(x,t,r,p)\cdot p -H(x,t,u,p) = a(x,t)(q-1)|p|^q -b(x,t)\cdot p + f(x,t)\; .$$
    and in order to verify \hyp{BA-HJ-U}, we have to reinforce the convexity assumption by assuming
    $a(x,t)>0$ in $\R^N\times [0,\Tf]$.  If $B$ is a ball in $\R^N\times [0,\Tf]$, we set
    $m(B)=\min_B a(x,t)$ and we have, using Young's inequality
    $$ D_p H(x,t,r,p)\cdot p -H(x,t,u,p) = m(B)(q-1)|p|^q + 1 -L(B)\; .$$
    Here the ``$+1$'' is just a cosmetic term to be able to set $G(p):= m(B)(q-1)|p|^q+1\geq 1$ and
    $L(B)$ is a constant depending on the $L^\infty$-norm of $b$ and $f$ on $B$.

    Finally, $a, b, f$ being locally Lipschitz continuous, it is clear enough that the estimates on
    $|D_x H(x,t,r,p)|, |D_t H(x,t,r,p)|$ hold. It is worth pointing out that the behavior at
    infinity of $a,b,f$ does not play any role since we have the arguments of the comparison proof
    are local. But, of course, we do not pretend that this strategy of proof is optimal...

    The checking of \HSubHJ is more ``example-dependent'' and we are not going to try to find ``good
    frameworks''. If $b=0$ and if there exists $\eta >$ such that 
    $$ \eta \leq a(x,t)\leq \eta^{-1}  \quad \hbox{in  }\R^N\times (0,\Tf)\, ,$$
    the Oleinik-Lax Formula suggests subsolutions of the form
    $$ \psi(x,t) = -\alpha (t + 1)( |x|^{q'}+1) - \beta\; ,$$
    where $q'$ is the conjugate exponent of $q$, i.e. $\displaystyle \frac1q + \frac1{q'}=1$ and
    $\alpha, \beta$ are large enough constants. Indeed 
    $$\psi_t + a(x,t)|D_x \psi|^q -f(x,t) \leq -\alpha ( |x|^{q'}+1) + 
    \eta^{-1}[q' \alpha (t + 1)]^q |x|^{q'} -f(x,t)\; .$$
    If there exists $c>0$ such that
    $$ f(x,t) \geq -c( |x|^{q'}+1)  \quad \hbox{in  }\R^N\times (0,\Tf)\, ,$$
    then, for large $\alpha$, namely $\alpha > \eta^{-1}[q' \alpha]^q + c$, one has a subsolution
    BUT only on a short time interval $[0,\tau]$. Therefore one has a comparison result if, in
    addition, the initial data satisfies for some $c'>0$ 
    $$ u_0(x) \geq -c'( |x|^{q'}+1)  \quad    \hbox{in  }\R^N\, ,$$
    in which case, we should also have $\alpha >c'$.

    In the good cases, the comparison result on $[0,\tau]$ can be iterated on $[\tau,2\tau]$,
    $[2\tau,3\tau]$, etc. to get a full result on $[0,\Tf]$.
\end{example}

\section{Whitney stratifications}
\label{sect:whitney}
\index{Whitney stratification}

There are mainly two reasons for introducing stratifications in dealing with discontinuities. On one
hand we may want to solve different equations on different submanifolds---or \emph{strata}---of the
stratification and make them work coherently; on the other hand we can consider a general
Hamilton-Jacobi equation (or control problem) posed everywhere, but presenting some discontinuities
located on the strata. Essentially, both questions are two different ways of looking a the same
reality.

Now, before going further, let us mention that in this book we use several concepts of
stratifications, labelled as
\begin{enumerate}
    \item \emph{General Stratifications}, which is the closest to the general concept of
stratifications in the sense of Whitney.
\item \AFS for \emph{Admissible Flat Stratifications}, where the strata are given by affine
subspaces, a particularly simple example of stratification.
\item \LFS for \emph{Locally Flattenable Stratifications}, which are stratifications that can
be locally reduced to an \AFS through a diffeomorphism.
\item \TFS for \emph{Tangentially Flattenable Stratifications}, where the flattening can be relaxed,
extending the notion of \LFS to situations involving some cusps for instance. This last notion of
stratification is really the one that is needed to make our methods work. 
\end{enumerate}

Let us now rapidly review where each type of stratification is used.

In \cite{BH}, Bressan and Hong study Hamilton-Jacobi-Bellman Equations and control problems
with discontinuities in the case when these discontinuities form a Whitney stratification, \ie when they
satisfy the Whitney conditions found in \cite{W1,W2}. 

In \cite{AEYW}, the more restrictive notions of \AFS and \LFS are introduced as well-adapted
structures to deal general discontinuities\footnote{The terminology is slightly different in
\cite{AEYW}}. We recall below this approach, and we also describe the restrictions these notions
impose on the Whitney stratification.

Finally we define the more general notion of \TFS, which turns out to be the most suitable framework
for setting up the methods we use throughout this book---especially in Part~\ref{stratRN}.

Before we begin, notice that, for the moment, we consider stratifications in $\R^N$ but 
\begin{enumerate}
    \item[$(i)$] since the various definitions of \AFS and \LFS are purely local, such stratifications of an open
    subset $\OO \subset \R^N$ can be defined exactly in the same way. We will do it for the \TFS.
    \item[$(ii)$] When considering time-dependent problems, we have to consider
        stratifications in $\R^{N+1}$---or more precisely of $\R^N\times (0,\Tf)$---, adding one dimension for time and
        using the remark of Point $(i)$. This allows
        to treat the case of time-depending stratifications, 
        see Chapter~\ref{chap:strat-def}.
    \item[$(iii)$] Stratifications can also be considered in a closed set, typically the closure of a domain $
    \Omega\subset\R^N$. In this case, as we will see in Part~\ref{S-BC}, both the interior of the set and
    the boundary---typically $\Omega$ and $\partial\Omega$---can be stratified. Of course, this last point can also
    be combined with $(i)$ and $(ii)$ and this is what we will do in Part~\ref{S-BC}, looking at stratifications of
    $\Omegb\times (0,\Tf)$.
\end{enumerate}

\subsection{General and admissible flat stratifications}
\label{sect:afs}

The notion of stratification we consider follows those introduced in Bressan and Hong~\cite{BH} but
the different parts of the stratification are not organized in the same way. Here
we assume that
$$\R^N=\Man{0}\cup\Man{1}\cup\cdots\cup\Man{N}\; ,$$
where the $\Man{k}$ ($k=0..N$) are disjoint $k$-dimensional submanifolds of $\R^N$. While, in~\cite{BH},
only a finite number of $\Man{k}$ are considered--or with our convention, the $\Man{k}$ can only
have a finite number of connected components--, here Definition~\ref{def:GS} states that each
$\Man{k}$ has only a locally finite number of connected components. We will write this decomposition
of $\Man{k}$ in connected components as
$$\Man{k}=\bigcup_{i\in\mathcal{I}_k }\Man{k}_i\, ,$$
where $\mathcal{I}_k$ are finite or countable sets. The $\Man{k}_i$ are called the ``stratas''.
In other words, we gather in $\Man{k}$ all the stratas which have the same dimension. 

Let us begin with the definition of a general stratification. 
\index{Stratification!general}

\begin{definition}\label{def:GS}\emph{--- General Stratifications.}\smsp
We say that $\M=(\Man{k})_{k=0..N}$ is a \emph{General Stratification} of $\R^N$
if the following set of hypotheses \HSTGEN is
satisfied\label{page:HSTgen}
\begin{enumerate}
    \item[$(i)$] For any $k=0..N$, $\Man{k}$ is a $k$-dimensional submanifold of $\R^N$.
    \item[$(ii)$] If $\Man{k}_i \cap \overline {\Man{l}_j}\neq \emptyset$ for some $l>k$
       then $\Man{k}_i \subset \overline {\Man{l}_j}$.
    \item[$(iii)$] For any $k=0..N$, $\overline{\Man{k}} \subset
       \Man{0}\cup\Man{1}\cup\cdots\cup\Man{k}$. 
    \item[$(iv)$] If $x \in \Man{k}$ for some $k=0..N$, there exists $r=r_x>0$ such that
    \begin{enumerate}
    \item[$(a)$] $B(x,r) \cap \Man{k}$ is a {\em connected} submanifold of $\R^N$;
     \item[$(b)$]  For any $l<k$, $ B(x,r) \cap \Man{l}=\emptyset$\,;
     \item[$(c)$]  For any $l>k$, $B(x,r)\cap\Man{l}$ is either empty or has at most a finite
                number of connected components\,;
    \item[$(d)$] For any $l>k$, $B(x,r)\cap\Man{l}_j\neq\emptyset$ if and only if
                $x\in\partial\Man{l}_j$.
        \end{enumerate}
    \end{enumerate}
\end{definition}

We point out that, even if the formulation is slightly different, and forgetting the number of connected components
of each $\Man{k}$, Assumptions \HSTGEN are
equivalent to the assumptions of Bressan and Hong \cite{BH}. Indeed, we both assume that we have a partition of
$\R^N$ with disjoints submanifolds but, as we already mention it above, we define a different way
the submanifolds $\Man{k}$. The key point is that for us $\Man{k}$ is here a $k$-dimensional
submanifold while, in \cite{BH}, the $\Man{j}$ can be of any dimension.  In other words, {\em our}
$\Man{k}$ is the union of all submanifolds of dimension $k$ in the stratification of Bressan and Hong.

With this in mind it is easier to see that our assumptions  \HSTGEN-$(ii)$-$(iii)$
are equivalent to the following assumption of Bressan and Hong: if $\Man{k}
\cap \overline {\Man{l}} \neq \emptyset$  then $\Man{k} \subset
\overline{\Man{l}}$  for all indices  $l,k$ without asking $l>k$ in our case.
But according to \HSTGEN-$(iii)$, as we already mention it above, $\Man{k} \cap \overline
{\Man{l}} = \emptyset$ if $l<k$: indeed for any $x\in \Man{k}$, there
exists $r>0$ such that $B(x,r) \cap \Man{l}=\emptyset$. This property
clearly implies \HSTGEN-$(iv)(b)$.

On the other hand, Assumption \HSTGEN-$(iv)(d)$ is just a consequence of $(iv)(c)$ provided we choose
the radius $r_x>0$ small enough. Indeed, since, by $(iv)(c)$, there is only a finite number
of connected components $\Man{l}_j$ for $l>k$ such that $B(x,r)\cap\Man{l}_j\neq\emptyset$,
we can exclude all those such that $\dist(x,\Man{l}_j)>0$ by choosing a smaller radius $r$.
        
Finally Condition \HSTGEN-$(iv)(a)$ implies that the set $\Man{0}$, if not void, consists
of isolated points. 

A specific and important case of stratification satisfying \HSTGEN is when the strata are flat,
\ie they all reduce to portions of vector spaces in $\R^N$. We call such stratifications \AFS, for
\emph{Admissible Flat Stratifications}.

To state a precise definition, we use the notations: 
\begin{enumerate}
 \item for $k=0..N$, $\V{k}$ is the set of all $k$-dimensional affine subspaces of $\R^N$;
 \item For $x\in\R^N$, $\V{k}(x)\subset\V{k}$ is the subset of affine subspaces containing $x$. In
     other words, $V\in\V{k}(x)$ if $V=x+V_k$ where $V_k$ is a $k$-dimensional vector subspace of
        $\R^N$. 
\end{enumerate}

\index{Stratification!admissible flat}
\begin{definition}\label{def:AFS}\emph{--- Admissible Flat Stratifications.}\smsp
    The stratification $\M$ is an \AFS if it satisfies \HSTGEN, with the exception of 
    property \HSTGEN-$(iv)(a)$, which is replaced by \label{not:AFS}\label{page:HSTflat} 
    \begin{equation*}\label{not:Vk} 
        \text{\HSTFLAT-}(iv)(a)\quad B(x,r) \cap \Man{k} = B(x,r) \cap (x+V_k) \text{ for some
        }(x+V_k)\in\V{k}(x)\;.
    \end{equation*}
    We denote by \HSTFLAT the set of conditions $(i)-(iv)$ with this replacement.
\end{definition}

Before providing several useful properties of \AFS, we consider several examples, the first one being the simplest
relevant example of a flat stratification.
\begin{example}
    We consider in $\R^2$ a chessboard-type configuration, see Figure~\ref{fig:chess}. In this case,
    we have the following decomposition: $$\Man{0}= \Z \times \Z\; ,$$
    $$ \Man{1}= \Big\{(\Z\times \R) \cup (\R \times \Z)\Big\}\setminus \Z^2\; ,$$
    and $\Man{2}=\R^2 \setminus (\Man{0}\cup \Man{1})$. In this simple case, the checking of the
    \HSTFLAT-assumptions is straightforward.  
\begin{figure}[htp]
   \begin{center}
   \includegraphics[width=0.8\textwidth]{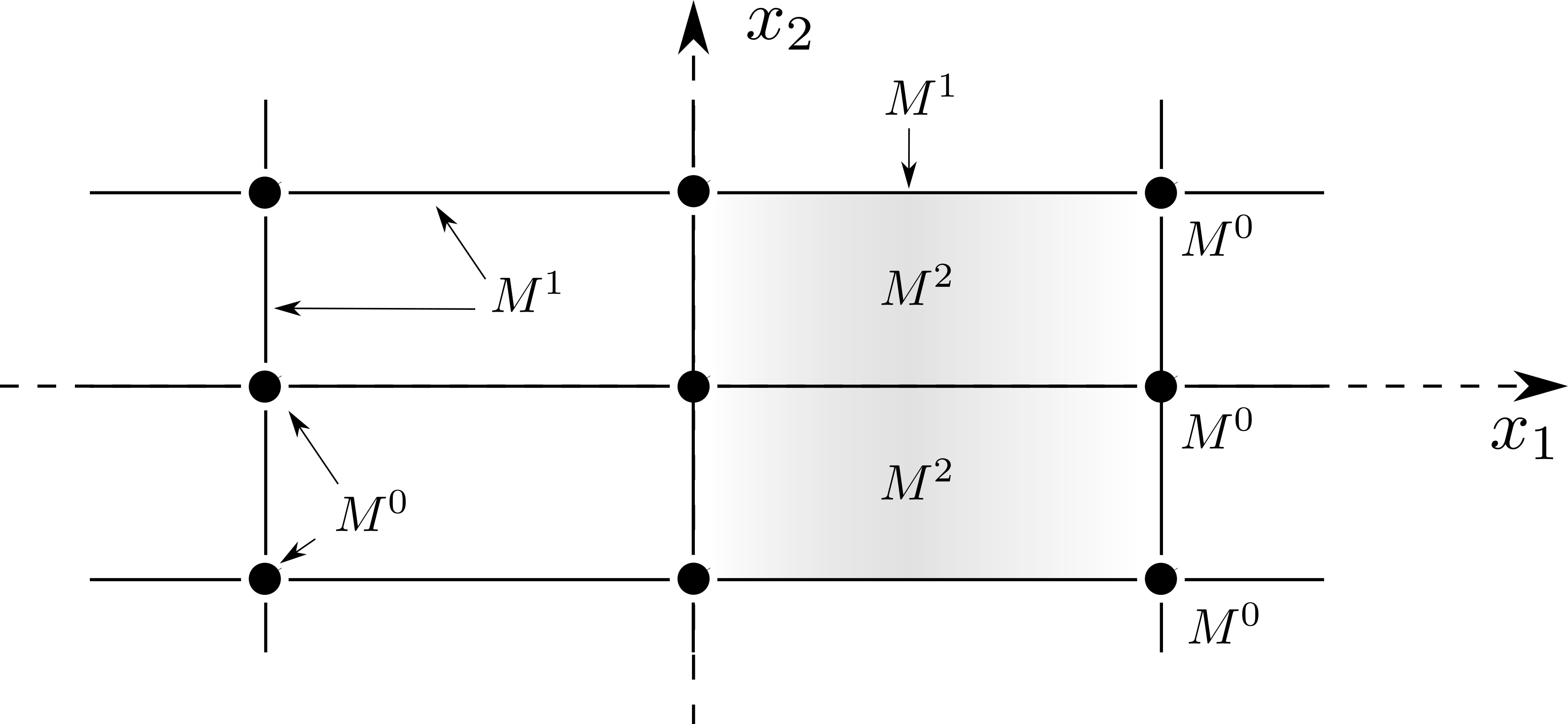}
   \caption{The chessboard-type configuration}
   \label{fig:chess} 
   \end{center}
\end{figure}
\end{example}

To emphasize the difference between ``flat configurations'' which are (or not) an \AFS, we propose
the example
\begin{example}
Let us consider a flat stratification in $\R^3$ induced by
the upper half-plane $\{x_3>0,x_2=0\}$ and the $x_2$-axis (see
figure~\ref{fig:strat.ex}).

\begin{figure}[htp]
   \begin{center}
   \includegraphics[width=0.6\textwidth]{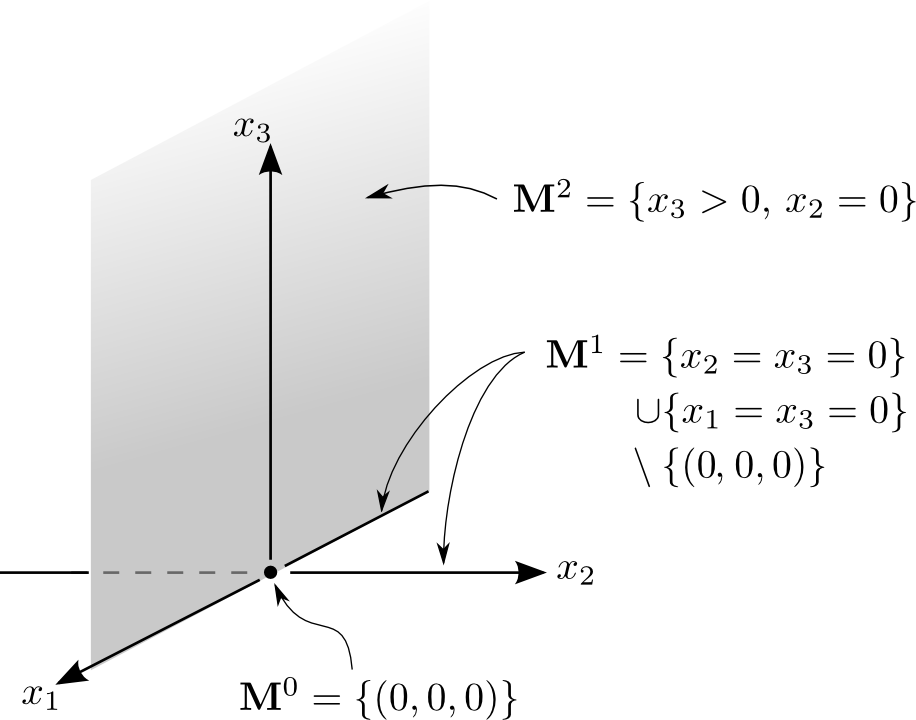}
   \caption{Example of a 3-D stratification}
   \label{fig:strat.ex} 
   \end{center}
\end{figure}

\noindent \textbullet\ \emph{The ``good'' stratification} consists in 
setting first $\Man{2}=\{x_3>0,x_2=0\}$. By \HSTFLAT-(iii), the boundary of $\Man{2}$ which is the
$x_1$-axis is included in $\Man{1}\cup\Man{0}$ and we also have $x_2$-axis in the stratification. In
this case, $\Man{1}\cup\Man{0}$ is the cross formed by the $x_1$ and $x_2$-axis but in order for
$\Man{1}$ to be a manifold, $(0,0,0)$ has to be excluded and we have to set here
$\Man{0}=\{(0,0,0)\}$. Thus, $\Man{1}$ consists of four connected components which are induced by
the $x_1$- and $x_2$-axis (but excluding the origin, which is in $\Man{0}$). Notice that in this
situation, the $x_3$-axis has no particular status, it is included in $\Man{2}$.

\noindent \textbullet\ \emph{A wrong approach} would be the following alternative decomposition:
$$ \Man{2} =\{x_3>0,x_2=0\},\  \Man{1} =\{x_1=x_3=0\} \cup \{
    x_2=x_3=0\},\  \Man{3}=\R^3-\Man{2}-\Man{1}\;.$$
Because $(0,0,0)\in \Man{1} \cap\, \overline {\Man{2}}$ but clearly $\Man{1}$ is not
included in $\overline {\Man{2}}$, so \HSTFLAT-$(ii)$ forbids this decomposition of $\R^3$.
\end{example}

Now we study the properties of \AFS.

\begin{lemma}\label{flat-lem.0} 
    Let $\M=(\Man{k})_{k=0..N}$ be an \AFS of $\R^N$. Then, for any $k=0..N$ and $i\in\mathcal{I}_k$,
    there exists an open set $\cO=\cO(i,k)\subset\R^N$ and $\V{k}_i\in\V{k}$ such that
    $$\Man{k}_i=\cO\cap\V{k}_i\;.$$
    In other words, there exists a $k$-dimensional vector space $V^i_k$ such that for any $x\in \Man{k}_i$, 
    $\Man{k}_i=\cO\cap (x+ V^i_k)$.
\end{lemma}
\begin{proof}
    Let $k\in\{0,..,N\}$, $i\in\mathcal{I}_k$, and fix $x\in\Man{k}_i$. By \HSTFLAT-$(iv)(a)$, for
    any $z\in \Man{k}_i$, there exists $\V{k}_{i(z)} \in \V{k}$ such that $z\in \V{k}_{i(z)}$.
    Now, consider the function
    $$\begin{aligned} \chi :\quad & \Man{k}_i\to\{0,1\}\\[2mm]
        & z\mapsto \begin{cases} 1 \text{ if }\V{k}_{i(z)}=\V{k}_{i(x)}\;,\\
            0 \text{ otherwise.}
    \end{cases}\end{aligned}$$
    This function is obviously locally constant: indeed, by \HSTFLAT-$(iv)(a)$, if $z\in\Man{k}_i$
    then $B(z,r_z)\cap \Man{k}_i=B(z,r_z)\cap \V{k}_{i(z)}$ and therefore if $z'\in B(z,r_z)\cap
    \Man{k}_i$, necessarily $\V{k}_{i(z')}=\V{k}_{i(z)}$.
        
    Therefore, since $\Man{k}_i$ is connected, it follows that
    $\chi$ is in fact constant, so that $i(z)=i(x)=i$ for all $z\in\Man{k}_i$. In other words,
    \HSTFLAT-$(iv)(a)$ can be written for all $z\in\Man{k}_i$ by means of only one affine subspace
    $$B(z,r_z)\cap \Man{k}_i=B(z,r_z)\cap\V{k}_{i}\;.$$
    We then set $\cO:=\cup_{z\in\Man{k}_i}B(z,r_z)$ which is an open set in $\R^N$. We deduce
    from the previous set equality that $\cO\cap\Man{k}_i = \cO\cap\V{k}_{i}$.
\end{proof}

As a consequence of the definition we have following result which will be
useful in a tangential regularization procedure (see Figure~\ref{fig.1} below) but that we will generalize through the
notion of tangentially flattenable stratification.

\begin{lemma}\label{flat-lem} 
    Let $\M=(\Man{k})_{k=0..N}$ be an \AFS of $\R^N$. Let $x\in\Man{k}$ and $r>0$, $V_k$ be as in
    \HSTFLAT-$(i)$.  If $y \in B(x,r) \cap \Man{l}_j$ for some $l>k$ and $j\in \mathcal{I}_l$ 
    then $x\in\overline{\Man{l}_j}$ and 
    $$B(x,r) \cap (y+V_k) \subset B(x,r)\cap\Man{l}_j\;.$$
\end{lemma}

\begin{proof} 
    Notice that by \HSTFLAT-$(iv)(d)$ we already know that $x\in\overline{\Man{l}_j}$, but moreover
    \HSTFLAT-$(ii)$ implies that $\Man{k}_i\subset \overline{\Man{l}_j}$. Using open sets $\cO_k$
    and $\cO_l$ defined in Lemma~\ref{flat-lem.0} we get 
    $$\cO_k\cap(x+V_k)\subset\overline{\cO_l}\cap(y+V_l)=\overline{\cO_l}\cap(x+V_l)\;,$$
    the last equality being justified by the fact that $x\in\partial\Man{l}_j$.

    This implies that $V_k$ is a subspace of $V_l$, so that clearly for any $y\in\Man{l}_j$,
    $y+V_k\subset y+V_l$. The result directly follows after intersecting with $B(x,r)$.
\end{proof}

\begin{figure}[htp]
    \label{fig.1} 
    \begin{center}
   \includegraphics[width=0.6\textwidth]{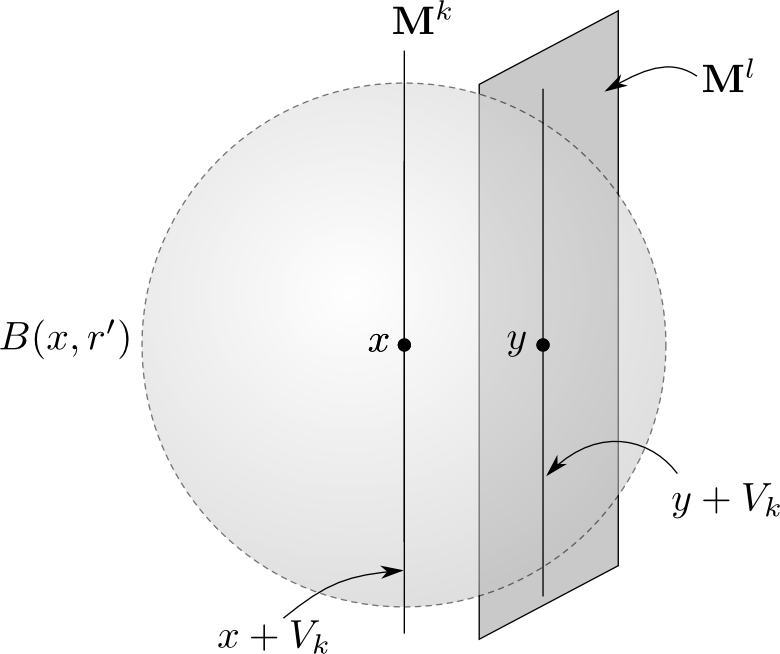}
   \caption{local situation}
   \end{center}
\end{figure}

\begin{remark}\label{rem:whitney}
    In this flat situation, the tangent space of $\Man{k}$ at $x$ is $T_x:=x+V_k$ while the tangent
    space of $\Man{l}$ at $y$ is $T_y:=y+V_l$, where $l>k$. The previous lemma implies that if
    $(y_n)_n$ is a sequence converging to $x$, then the limit tangent plane of the $T_{y_n}$ is
    $x+V_l$ and it contains $T_x$, which is exactly the Whitney condition---see {\rm
    \cite{W1,W2}}.
\end{remark}

As we will see it below, an \AFS is a perfect framework where our methods fully apply, in particular
because of Lemma~\ref{flat-lem}. And clearly a similar remark holds for all the stratifications
which can locally be reduced to \AFS through a smooth enough diffeomorphism; this leads us to
introduce, in the next section, the notion of {\em Locally Flattenable Stratification \LFS}. But
Section~\ref{sect:LRSA} provides some properties of the \LFS which shows that a general
stratifications is not, in general, a \LFS. Intuitively the reader should realize that
``flattening'' locally all the $\Man{k}$ imposes rather rigid conditions on the ``geometry'' of a
stratification and we do not know checkable conditions or characterizations which would allow to
decide whether a given stratification is a \LFS or not.  On an other hand, a more adapted concept to
our approach, which we call {\em Tangentially Flattenable Stratifications \TFS}, consists in looking
at stratifications which satisfy Lemma~\ref{flat-lem} after a suitable change of coordinates. This
notion is more general than the \LFS-one but still the same remark holds: we do not know checkable
conditions or characterizations which would allow to decide whether a given stratification is an
\TFS or not.

\subsection{Locally flattenable stratifications \LFS}
\index{Stratification!locally flattenable}

Particular---yet quite representative---cases of general stratifications can be obtained by
smooth enough modifications of flat stratifications.
\begin{definition}\label{def:LFS}\label{page:HSTlfs}\emph{--- Locally Flattenable
    Stratifications.}\smsp
    We say that $\M=(\Man{k})_{k=0..N}$ is a locally flattenable stratification \label{page:HSTreg} of
    $\R^N$---\,\LFS in short---\,if it satisfies the two following assumptions denoted by \HST[\rm LFS]
    \begin{enumerate}
        \item[$(i)$] the following decomposition holds:
            $\R^N=\Man{0}\cup\Man{1}\cup\cdots\cup\Man{N}$; 
        \item[$(ii)$] for any $x \in \R^N$, there exists $r=r(x)>0$ and a $C^{1,1}$-change of
            coordinates $\Psi^x: B(x,r) \to \R^N$ such that $\Psi^x(x)=x$ and
            $\{\Psi^x(\Man{k}\cap B(x,r))\}_{k=0..N}$ is the restriction to $\Psi^x(B(x,r))$ of an
            \AFS in $\R^N$.
    \end{enumerate}
\end{definition}

We point out that it is easy to check that a \LFS satisfies \HSTGEN and therefore is a general
stratification in the sense of Definition~\ref{def:GS}; indeed, all the properties of a general
stratification are local and the way a \LFS is defined, the diffeomorphisms $\Psi^x$ transfer all
the local property of an \AFS, in particular the  \HSTGEN-ones.

\begin{remark}
    If we need to be more specific, we also say that $(\M,\Psi)$ is a
    stratification of $\R^N$, keeping the reference $\Psi$ for the collection of
    changes of variables $(\Psi^x)_x$. This will be usefull in
    Section~\ref{NESR} when we consider sequences of stratifications.
\end{remark}

\noindent\textbf{Tangent spaces --}
The definition of locally flattenable stratifications (flat or not) allows to define, for
each $x\in\Man{k}$, the tangent space to $\Man{k}$ at $x$, denoted by
$T_x\Man{k}$. 
To be more precise, if $x\in\Man{k}$ and $r>0,\ V_k$ are as in \HSTFLAT-$(iv)$, then
$$T_x\Man{k}=(D\Psi^x (x))^{-1}(V_k)\;,$$ which can be identified to $\R^{k}$. 
Moreover, we can decompose $\R^N = V_k
\oplus V_k^\bot$, where $V_k^\bot$ is the orthogonal space to $V_k$. For any
$p \in \R^N$,  we have $p= p_\top + p_\bot$ with $p_\top \in V_k$ and
$p_{\bot}\in V_k^\bot$. In the special case $x\in\Man{0}$, we have $V_0=\{0\}$,
$p=p_\bot$ and $T_x\Man{0}=\{0\}$.

\bigskip

The notion of stratification is introduced above as a pure geometrical tool and
it remains to connect it with the singularities of Hamilton-Jacobi Equations.
Our aim is to define below a ``natural framework'' allowing to treat
Hamilton-Jacobi Equations (or control problems) with discontinuities, which
will involve two types of information: some conditions on the kinds of
singularities we can handle and some assumptions on the Hamiltonians in a
neighborhood of those singularities.

We provide here a first step in this direction by considering the simple
example of an equation set in the whole space $\R^N$ $$ H(x,u,Du)=0\quad
\hbox{in  }\R^N\; ,$$ where the Hamiltonian $H$ has some discontinuities (in
the $x$-variable) located on some set $\Gamma\subset\R^N$. The first question
is: what kind of sets $\Gamma$ can be handled?

The approach we systematically use consists in assuming that $\Gamma$ provides a stratification
$\M=(\Man{k})_{k=0..N}$ of $\R^N$. This means that $\Man{N}$ is the open subset of $\R^N$ where $H$
is continuous while $\Man{k}$ contains the discontinuities of dimension $0\leq k\leq (N-1)$.  Of
course, some of the $\Man{k}$ can be empty.

What should be done next is to clarify the structure of the Hamiltonian $H$ in
a neighborhood of each point $x\in \Man{k}$ and for each $k\leq (N-1)$. This is
where the previous analysis on stratifications allows to reduce locally the
problem to  the following situation: if $x\in\Man{k}$, there is a ball
$B(x,r)$ for some $r>0$, and a $C^1$-diffeomorphism $\Psi$ such that 
$$B(x,r)\cap\Psi(\Man{k})=B(x,r)\cap\bigcup_{j=0}^k \big(x+V_{j}\big)\;.$$
In other words, through a suitable $C^1$ change of coordinates, we are in a
\emph{flat} situation where $x$ is only possibly ``touched'' by
$j$-dimensional vector spaces for $j\geq k$.

\subsection{Limits of the \LFS approach}\label{sect:LRSA}

The notion of locally flattenable stratification seems to provide a very general framework in which
one could think that many situations can be treated. As we have seen, several quite special
geometric structures can be handled, corresponding to a great variety of discontinuities in the
equations we consider.

However, there are very simple situations that the stratified framework cannot handle. Let us focus
here on curves in $\R^2$ in order to better understand the problems that may occur.

The major restriction that stems directly from the very definition of \LFS is
that locally, all the elements of the stratification have to be flattenable \emph{simultaneously}.
We come back later on how the notion of \TFS allows to relax this hypothesis but let us mention that
this leaves out the following example: consider in $\R^2$ a continuous curve
$\gamma:(0,1)\to\R^2$ having an infinite length (near $s=0^+$), such that $\gamma(0^+)=(0,0)$. The natural
stratification associated to this situation is 
$$\Man{0}=\{(0,0)\}\;,\quad \Man{1}=\{(s,\gamma(s)):s\in(0,1)\}\;,\quad
\Man{2}=\R^2\setminus(\Man{0}\cup\Man{1})\;.$$
But we clearly see that locally around $(0,0)$, the \LFS condition cannot hold, 
otherwise $\overline{\Man{1}}=\{(s,\gamma(s))\}\cup\{(0,0)\}$ could be
flattened through a $C^{1,1}$-diffeomorphism, implying that the initial curve is of finite length.

Cusps are also the typical examples of geometric structures which cannot be included in \LFS:
consider the curve 
$$\Gamma:=\big\{y=\sqrt{|x|} : x\in\R\big\}\subset\R^2\;.$$
The natural (and only) stratification of $\Gamma$ would be to set
$$\Man{0}=\{(0,0)\}\;,\ \text{and}\  \Man{1}=\{y=\sqrt{-x}:x<0\}\cup
\{y=\sqrt{x}:x>0\}\;.$$
However, condition $(ii)$ of the locally flattenable stratification definition cannot hold.
More precisely, at the singular point $z=(0,0)$, there is no $C^{1,1}$ change of variables $\Psi^z$
which can transform the cusp into a flat stratification since such a change of variables could not
be even Lipschitz continuous.
\begin{figure}[htp]
    \begin{center}
   \includegraphics[width=0.47\textwidth]{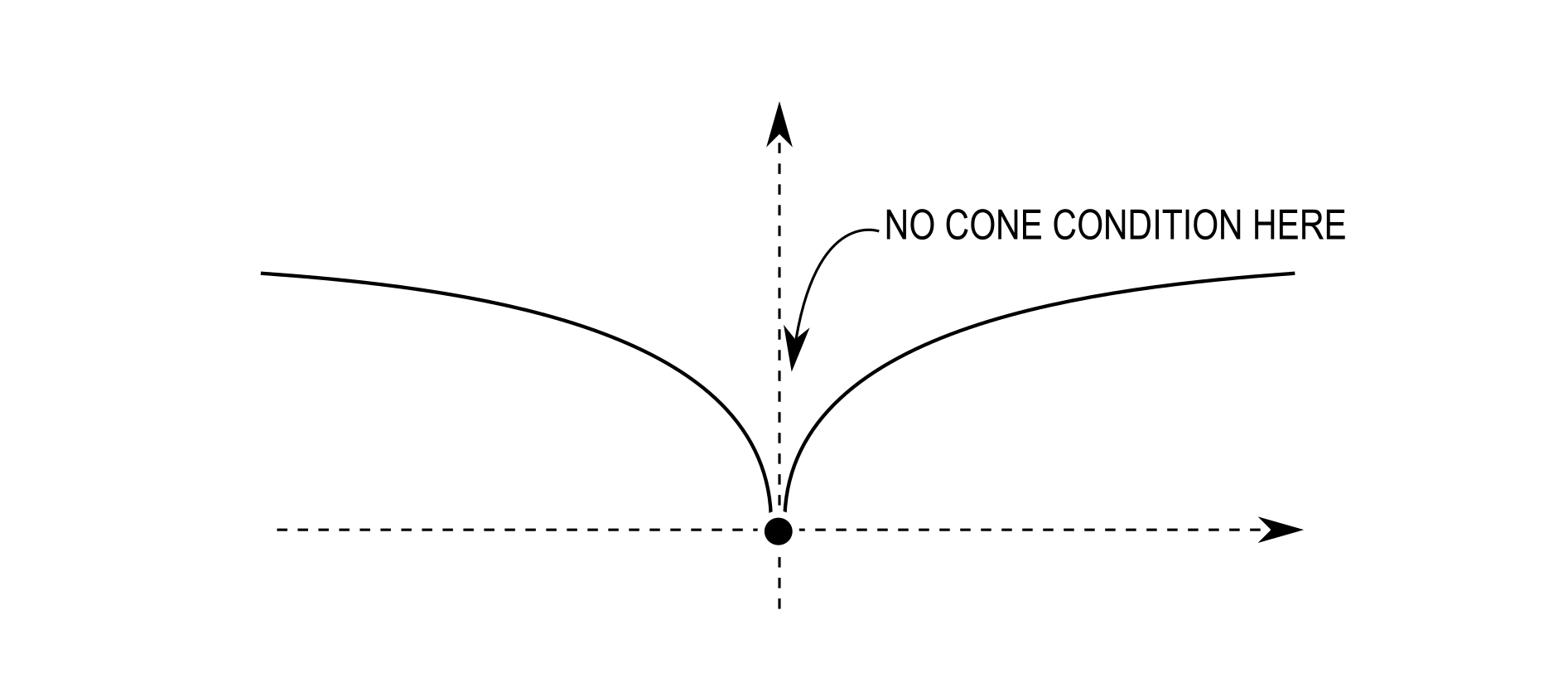}
    \includegraphics[width=0.47\textwidth]{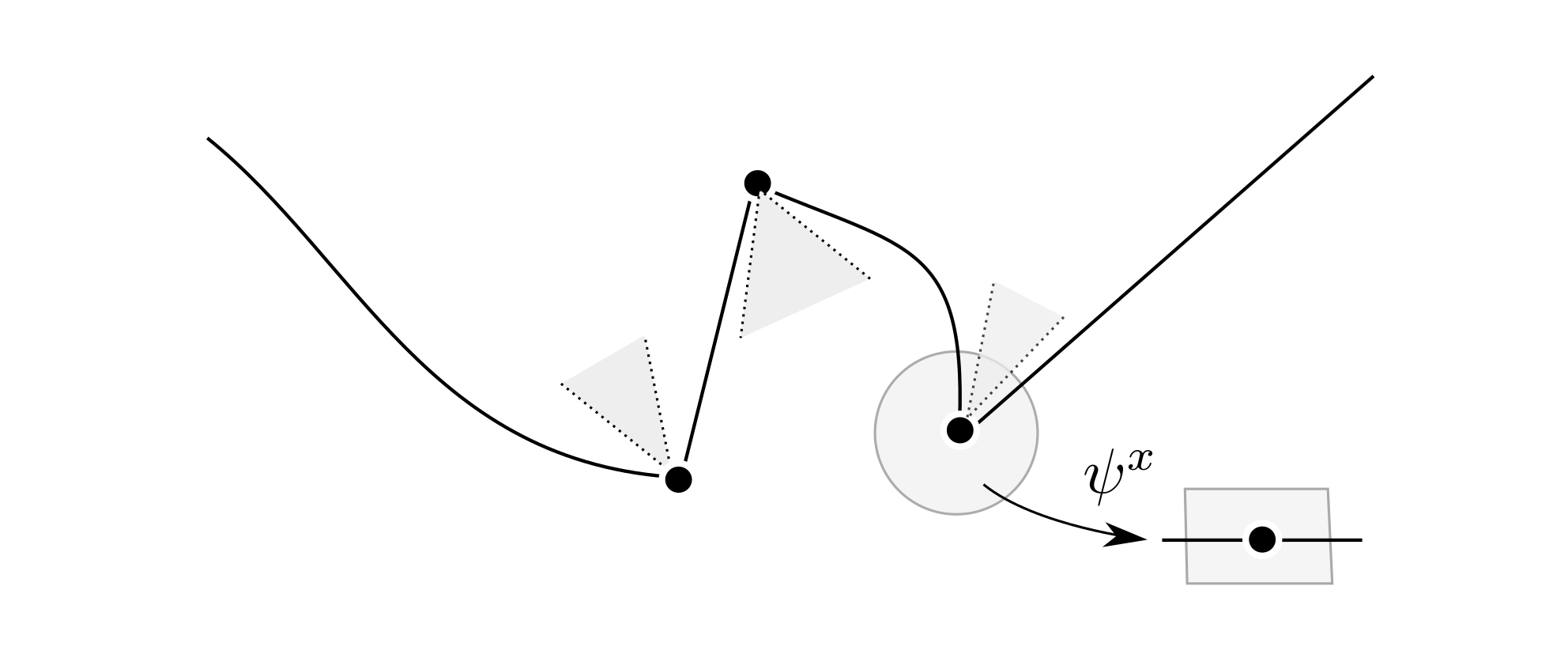}
        \caption{\emph{Left:} A cusp; \emph{Right:} A piecewise smooth example}
        \label{fig:cusp}\label{fig:cone-cond}
   \end{center}
\end{figure}

Nevertheless, a piecewise $C^1$ curve $\Gamma\subset\R^2$ satisfying a 
double-sided cone condition at junction points can always be considered 
as a locally flattenable stratification, after choosing $\Man{0}$
as the set of singular points. 
Indeed, if $x\in\Man{0}$, the $C^{1,1}$ diffeormorphism $\Psi^x$ just has to
``flatten the angle'' in order to get a flat stratification 
(see fig~\ref{fig:cone-cond}), which is of course possible.


In order to give a general result that \LFS must satisfy, we need to
introduce some objects. 

\bigskip

\noindent\textsc{Extended tangent spaces ---} 
    Let $x\in\Man{k}_i$ and $\Psi^x$, $B(x,r)$ as in the definition of \LFS.  If
    $x\in\partial\Man{l}_j$, then, combining Lemma~\ref{flat-lem} and Remark~\ref{rem:whitney},
    there exists a $l$-dimensional vector space $V_{l,j}$ such that $$ \Psi^x(\Man{l}_j \cap
    B(x,r))\subset (x+V_{l,j})\; ,$$ and we can extend the tangent space to $\Man{l}_{j}$ up to $x$
    by setting $$\overline{T}_x\Man{l}_{j}:=D(\Psi^x (x))^{-1}(V_{l,j})\;.$$

\


\noindent\textsc{Inward pointing cones ---} Let $\M$ be an \AFS and fix $x\in\Man{k}_i$ for some
$k\in\{0..N-1\}$, $i\in\mathcal{I}_k$. We assume that $x\in\partial\Man{l}_j$ for some $l>k$. We
first introduce the notion of \emph{inward directions to $\Man{l}_j$ at $x$}: a direction
$v\in\R^N\setminus\{0\}$ is said to point inward to $\Man{l}_j$ at $x$ if $x+hv \in \Man{l}_j$ for
$h>0$ small enough. Since $\Man{l}_j={\cO_{l,j}}\cap(x+V_{l,j})$ is flat, all these inward
directions $v$ belongs to $V_{l,j}$. Then we define the \emph{inward pointing cone}
$\Inf(l,j)(x)$ as the set containing all these inward directions to $\Man{l}_j$ at $x$. This
vector set is strictly positively homogeneous by definition and it does not contain
the tangential directions in $\partial\Man{l}_j$ nor $0$.

More generally, in the case of a \LFS the definition of the inward pointing cone is given by
 $$\In_x\Man{l}_j:=(D\Psi^x (x))^{-1}\Big(\Inf(l,j)(x)\Big) \subset \overline{T}_x\Man{l}_j\;.$$
Here also, the vectors in $\In_x\Man{j}_l$ are pointing strictly inwards $\Man{l}_j$,
excluding the directions tangent to $\partial\Man{l}_j$ at $x$ and $0$. Notice finally that since
$(D\Psi^x (x))^{-1}$ is linear, $\In_x\Man{l}_{j}$ is also strictly positively homogeneous.

An intrinsic characterization of the inward pointing cone can be given. To do so, for a given
$x\in\partial\Man{l}_j$, we consider the $C^1$-curves $\gamma:\R \to \R^N$ such that $\gamma(0)=x$
and $\gamma(s)  \in \Man{l}_j$ if $s\in (0,s_0)$ for some $s_0>0$. We will say that $\gamma \in
\Lambda^l_j(x)$ if there exists $\eta>0$ such that
\begin{equation}\label{eqn:carac-gamma}
    \dist(\gamma(s),\partial\Man{l}_j)\geq\eta s\quad \hbox{for all  }s\in(0,s_0)\;.
\end{equation}
Then the following characterization holds:
\begin{lemma}\label{lem:carac.cone}
    Given $x\in\Man{k}_i\cap\partial\Man{l}_j$, we have
    $\In_x\Man{l}_j=\big\{\dot\gamma(0):\gamma\in\Lambda^l_j(x)\big\}$.
\end{lemma}
\begin{proof}
    We first prove the result in the case of an \AFS. 

    \noindent \textbf{Direct inclusion --} 
    For the inclusion $\Inf(l,j)(x)\subset \In_x\Man{l}_{j}$, we have to show that if $v \in
    \Inf(l,j)(x)$ there exists $\eta >0$ such that $\gamma(s):= x+sv$ satisfies
    \eqref{eqn:carac-gamma} for $s\in (0,s_0)$, $s_0$ being small enough.  We argue by
    contradiction: if \eqref{eqn:carac-gamma} does not hold, there exists a sequence of positive
    numbers $s_\e\to0$ such that $$0<\dist(x+s_\e v,\partial \Man{l}_{j})\leq\e s_\e$$
    (of course the distance is positive because $x+s_\e v$ is in $\Man{l}_j$, not on its boundary).

    By \HSTFLAT-$(iv)$-$(c)$ and $(iii)$, we can extract a subsequence of $(s_\e)_\e$ (still
    denoted in the same way to simplify the exposure) such that the distance is achieved for $y_\e$
    in {\em the same} $\Man{n}_{m}$ for some $n<l$ and $m\in\mathcal{I}_n$.
    Hence, if $\Man{n}_{m}=(x+V_{n,m})\cap\cO_{n,m}$,
    $$ \Big|(x+s_\e v) - ( x+ w_\e) \Big|\leq \e s_\e\quad \hbox{for some  } w_\e\in V_{n,m}\; .$$
    We deduce from this property that $$\Big|v - \frac{w_\e}{s_\e}\Big|\leq \e \; ,$$ 
    and since $w_\e/s_\e\in V_{n,m}$ for any $\e>0$, by letting $\e$ tend to $0$ we deduce that $v\in
    V_{n,m}$. It follows that $x+s_\e v\in (x+V_{n,m})$ and thus, for $\e>0$ small enough, $x+s_\e
    v\in (x+V_{n,m})\cap\cO_{n,m}\subset\partial\Man{l}_j$ which contradicts 
    $\dist(x+s_\e v,\partial\Man{l}_j)>0$.  Hence \eqref{eqn:carac-gamma} is proved.

    \noindent \textbf{Converse inclusion --} 
    In order to prove that $\In_x\Man{l}_{j} \subset \Inf(l,j)(x)$, we take any $\gamma
    \in \Lambda^l_j(x)$ and we have to show that $\dot \gamma (0) \in \Inf(l,j)(x)$. Notice first
    that $\gamma (s) \in x+V_{l,j}$ for any $s\in (0,s_0)$ and therefore $\dot \gamma (0) \in
    V_{l,j}$. 
    On the other hand, by the differentiability of $\gamma$ at $0$,
    $$ \gamma(s)=x+ \dot \gamma (0)s +o(s)\; ,$$ and $x+ \dot \gamma (0)s \in x+V_{l,j}$.
    Now, by \eqref{eqn:carac-gamma} we see that for $s>0$ small enough,
    $$\dist(x+ \dot \gamma(0)s, \partial\Man{l}_j)\geq \dist(\gamma (s), \partial\Man{l}_j) + o(s)
    \geq (\eta +o(1)) s>0\;,$$
    which implies that $x+ \dot \gamma (0)s\in \Man{l}_j$ for any $s>0$ small enough. Hence
    $\dot\gamma(0)\in \Inf(l,j)(x)$ and we are done.

    \bigskip

    \noindent \textbf{The \LFS case --} 
    Here we use in an essential way the Lipschitz continuity of $\Psi^x$ and its
    $C^1$-property.

    If $v\in\In_x\Man{l}_j$, we claim that the curve $\gamma(s):=(\Psi^x)^{-1}(x+sv)$ belongs to
    $\Lambda^l_j(x)$: indeed, $\Psi^x(\gamma(s))=x+sv$ with $v\in \Inf(l,j)(x)$ and the first part
    of the proof implies that it satisfies \eqref{eqn:carac-gamma} for the (locally) flat
    stratification. Using the Lipschitz continuity of $\Psi^x$, we deduce that 
    $\gamma$ also satisfies \eqref{eqn:carac-gamma}, for some other parameters $\tilde\eta,
    \tilde s_0>0$.

    Conversely,  if $\gamma \in \Lambda^l_j(x)$, then the curve
    $\Gamma(\cdot):=\Psi^x(\gamma(\cdot))$ is also in the set $\Lambda^l_j(x)$ (but for the flat
    stratification) and therefore 
    $$\dot \Gamma (0) = (D\Psi^x)(x)\big(\dot \gamma(0)\big) \in  \Inf(l,j)(x)\;.$$
    By definition of $\In_x\Man{l}_j$, it follows that
    $\dot \gamma(0)\in \In_x\Man{l}_j$, and the proof is complete.
\end{proof}

\

The main result of this section is the
\begin{proposition}\label{prop:tangent.strat}
    Let $\M$ be a locally flattenable stratification of $\R^N$, $0\leq k<N$ and $x\in\Man{k}_i$ for some
    $i\in\mathcal{I}_k$.
    Assume that $x\in\partial\Man{l}_j\cap\partial\Man{l'}_{j'}$ for some $k<l,l'\leq N$.
    If $(l,j)\neq(l,j')$, then
    \begin{enumerate}
        \item[$(i)$] the following inclusion holds:
            $\Man{k}_i\subset\partial\Man{l}_j\cap\partial\Man{l'}_{j'}$\,;
        \item[$(ii)$] for any $x\in\Man{k}_i$, $\In_x\Man{l}_j\cap \In_x\Man{l'}_{j'}(x)=\emptyset$. 
    \end{enumerate}
\end{proposition}

Though this proposition is simple in its form, it rules out several cusp-like configurations
involving various dimensions (see below examples after the proof). 
In particular, in the case of the piecewise smooth curve in dimension $N=2$, we
recover that the tangents from both sides of a singular point cannot be equal
in the limit at such point. Notice that of course, they can possibly make
a $\pi$-angle but in that case, the inward pointing directions are opposite.

\begin{proof}
    Concerning $(i)$, the result follows directly from
    \HSTFLAT-$(ii)$:  since $$x\in\overline{\Man{l}_{j}}\cap\overline{\Man{l'}_{j'}}\;,$$ we get that
    $\Man{k}_i$ is included in both $\overline{\Man{l}_{j}}\cap\overline{\Man{l'}_{j'}}$. But since
    $l>k$, $\Man{k}_i$ does not intersect with $\Man{l}_{j}$ nor with $\Man{l'}_{j'}$, so that $(i)$
    holds.

    We now turn to $(ii)$ and consider first the case of an \AFS. Since $(l,j)\neq(l',j')$ then 
    $\Man{l}_j\cap\Man{l'}_{j'}=\emptyset$ which clearly implies that the inward pointing cones are
    disjoint. Indeed, as we noticed before, if $e\in\Inf(l,j)\cap\Inf(l',j')$ then for $h$ small
    enough, we get that $x+he\in\Man{l}_j\cap\Man{l'}_{j'}$ which is a contradiction. 

    In the \LFS case, the conclusion follows from the fact that since $D(\Psi^x (x))^{-1}$ is invertible
    it cannot map two different directions on the same one. More precisely, assume that 
    $$e\in\In_x\Man{l}_j\cap\In_x\Man{l'}_{j'}\neq\emptyset\;.$$
    Then there exist two  vectors $w\in\Inf(i,j)$,
    $w'\in\Inf(l',j')$ such that 
    $$e=D(\Psi^x (x))^{-1}(w)=D(\Psi^x (x))^{-1}(w')\;.$$
    But since $w\neq w'$ because they belong to $\Inf(l,j)$ and $\Inf(l',j')$ respectively, 
    we get a contradiction with the bijectivity of $D(\Psi^x (x))^{-1}$.
\end{proof}

\begin{remark}
    The fact that the cones $\In_x\Man{l}_j$ do not intersect implies that various cone conditions
    hold at $x$, separating the manifolds touching at this point (which are in finite number, see 
    \HSTFLAT-$(iv)(c)$). However, building explicitly such cones is quite difficult in all its
    generality and we wil not try to state it here. But notice that there is a lot of freedom in choosing 
    the directions of such cones: if $\In_x\Man{l}_j\cap\In_x\Man{l'}_{j'}=\emptyset$, any direction
    $e$ at positive distance from both cones allows to build a separating cone.
\end{remark}

\

\noindent\textsc{Typical situations ---} 
Of course very complex situations can occur involving different dimensions but let us see two simple
situations to understand the meaning of Proposition~\ref{prop:tangent.strat}-$(ii)$.

\

\noindent\textbf{Fig. \ref{fig:ex.cusp.1} ---} On the left the situation is allowed since at the point
$\{x\}=\Man{0}$, $\In_x\Man{2}_1=\Man{2}_1$, $\In_x\Man{2}_2=\Man{2}_2$ and therefore
$\In_x\Man{2}_1\cap\In_x\Man{2}_2=\emptyset$. Notice however that the boundaries intersect, which
corresponds to the direction of $\Man{1}_1$. 

On the right, it is clear that the problem does not come from $\In_x\Man{2}_{1/2/3}$ which do not
intersect (although $\In_x\Man{2}_{3}=\emptyset$), but from the $\Man{1}$ manifolds since
$\In_x\Man{1}_1\cap\In_x\Man{1}_2=\Man{1}_1\neq \emptyset$. This cusp-type situation is of course
not allowed.

\begin{figure}[htp]
   \begin{center}
   \includegraphics[width=0.9\textwidth]{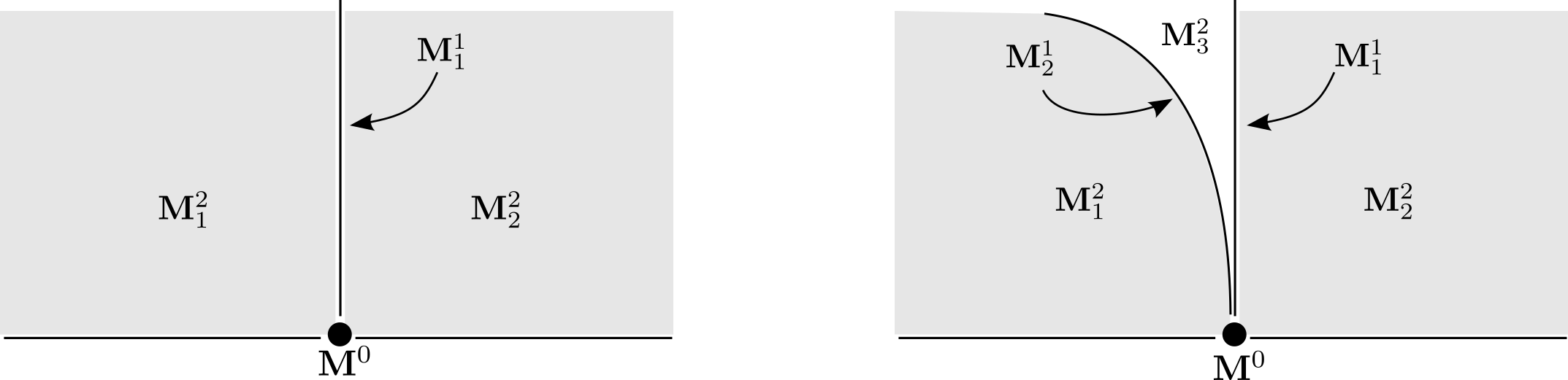}
       \caption{Examples in 2-D}
   \label{fig:ex.cusp.1} 
   \end{center}
\end{figure}

\noindent\textbf{Fig. \ref{fig:ex.cusp.2} ---} On the left the situation is allowed since the
semi-line $\Man{1}$ makes a non-zero contact angle with the plane $\Man{2}$. However, using for
instance the characterization in Lemma~\ref{lem:carac.cone} we see that $\In_x\Man{1}=\R^+_*e$, while $\In_x\Man{2}=\Man{2}$.  Hence $\In_x\Man{1}\cap\In_x\Man{2}=\R^+_*
e\neq \emptyset$, another cusp-type situation that is not allowed.

\begin{figure}[htp]
   \begin{center}
   \includegraphics[width=0.9\textwidth]{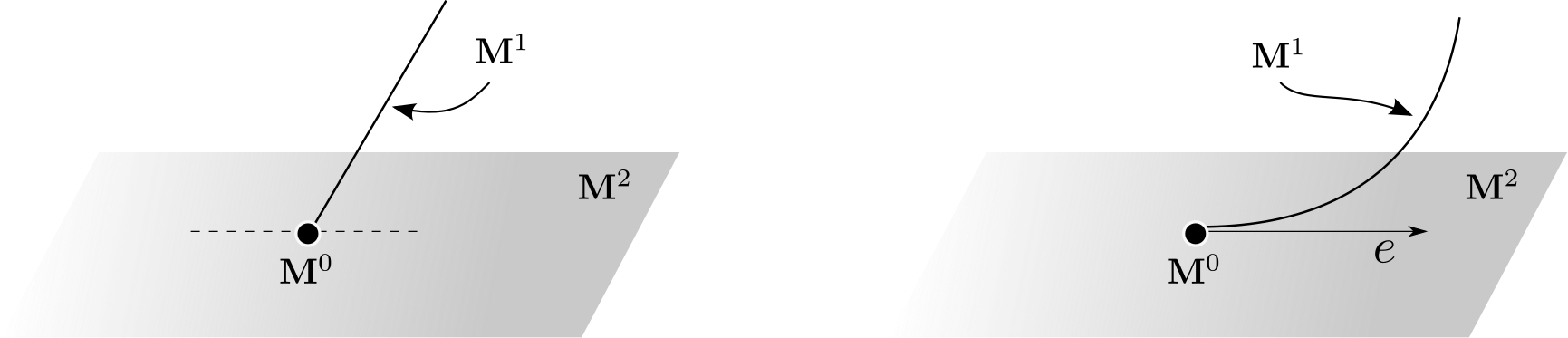}
       \caption{Examples in 3-D}
   \label{fig:ex.cusp.2} 
   \end{center}
\end{figure}

\

\noindent\textsc{Stratifications in domains ---} The question is whether we can
extend or not the notion of stratification in $\R^N$ to the case of open 
sets $\Omega$.

At this point of the book, we do not enter into details on this because we
devote a complete part of the book (Part~\ref{S-BC}) to the case of
state-constrained problems. Let us just mention that when we consider a domain
$\Omega$, its boundary $\partial\Omega$ has to be understood as a specific
part of the stratification. And if the boundary is not regular, we use the
stratified approach to decompose it in various manifolds of different
dimensions.

The conditions on the inward pointing cones that we proved above imply that $\Omega$ has to satisfy
a double-sided cone condition in order to deal with it in the stratified approach.  This cone
condition (at least the interior one) is also used in Section~\ref{abl} in order to get a suitable
boundary regularity for subsolutions.

\subsection{Tangentially flattenable stratifications \TFS}
\index{Stratification!tangentially flattenable}

As we have seen in the previous section, the notion of \LFS is quite restrictive:
it implies that, in a neighborhood of each point of $\R^N$, there exists a diffeomorphism
which flattens simultaneously every part of the stratification nearby. In fact, this property
turns out to be stronger than what we need. 

So, let us introduce finally the notion of {\em Tangentially Flattenable Stratification} which is less
restrictive, allowing to handle some situations where, for instance, cusps appear.  We even consider
the case of extended stratifications not only in $\R^N$, but in any domain $\OO\subset \R^N$: this
does not create any additional difficulty since every property is purely local.

\begin{definition}\label{def:TFS}\label{page:HSTtfs}\emph{--- Tangentially Flattenable
    Stratifications.}\smsp
    We say that $\M=(\Man{k})_{k=0..N}$ is a \emph{Tangentially Flattenable Stratification of $\OO$}\,---\TFS
    in short\label{not:TFS}---\,if the following hypotheses hold:
    \begin{enumerate}
        \item[$(i)$] Hypotheses \HSTGEN are satisfied;
        \item[$(ii)$] for any $k$, $\Man{k}$ is a $C^{1,1}$-submanifold of $\OO$; moreover, if $x\in \Man{k}$,
            there exists $r=r_x>0$ such that $B(x,r)\subset \OO$ and a $C^{1,1}$-diffeomorphism $\Psi_x$
            defined on $B(x,r)$ such that $\Psi_x(x)=x$ and
            $$\Psi_x(B(x,r) \cap \Man{k})=\Psi_x(B(x,r)) \cap (x+V_k)$$ 
            where $V_k$ is a $k$-dimensional vector subspace of $\R^N$;
        \item[$(iii)$] setting $\tMan{l}:=\Psi_x(B(x,r) \cap \Man{l})$ and
            $\tMan{l}_j= \Psi_x(B(x,r) \cap \Man{l}_j)$ for any connected component $\Man{l}_j$ of
            $\Man{l}$, 
            \begin{enumerate}
                \item[$(a)$] for any $l<k$, $ \tMan{l}=\emptyset$\,;
                \item[$(b)$] for any $l>k$, $\tMan{l}$ is either empty or has at most a finite
                number of connected components;
            \item[$(c)$] if $x\in\partial\tMan{l}_j$ and $y\in \tMan{l}_j$,
                $\Psi_x(B(x,r)) \cap (y+V_k) \subset \tMan{l}_j\;.$
            \end{enumerate}
    \end{enumerate}
    We denote by \HSTTFS this set of assumptions and we will say that a stratification which satisfies the same
    properties as $(\tMan{k})_{k=0..N}$ is {em tangentially flat}.
\end{definition}

The difference between a locally flattenable stratification and a tangentially flattenable one is that, in the \TFS case,
$(\tMan{l})_{l=0,..N}$ is not necessarily the restriction of an \AFS to $B(x,r)$, hence the
$\tMan{l}$ for $l\neq k$ are not necessarily affine spaces. They just have to be ``tangentially
flat'' thanks to \HSTTFS-$(ii)$ and $(iii)$-$(c)$. This property is the one we need in particular to perform the
tangential regularization described in the next section, while flattening all the stratification at
the same time is not a requirement.

\

In particular, the following stratification in $\R^3$ is an \TFS but not a \LFS:
$$ \Man{0}=\emptyset \quad,\quad \Man{1}=\{(x_1,x_2,x_3):\,x_2=0,x_3=0\}\; ,$$
$$\Man{2}=\{(x_1,x_2,x_3):\,x_2\neq 0,|x_3|=x_2^2\}\;,$$
and $\Man{3}=\R^3\setminus(\Man{0}\cup \Man{1}\cup\Man{2})$. In checking that this is an \TFS, only
Condition~$(iii)$ may cause a problem but it is more than clear here that it is satisfied. On the
other hand, $\Man{2}$ forms a cusp on $\Man{1}$ and therefore this cannot be a \LFS.

\begin{figure}[!htp]
    \begin{center}
   \includegraphics[width=0.85\textwidth]{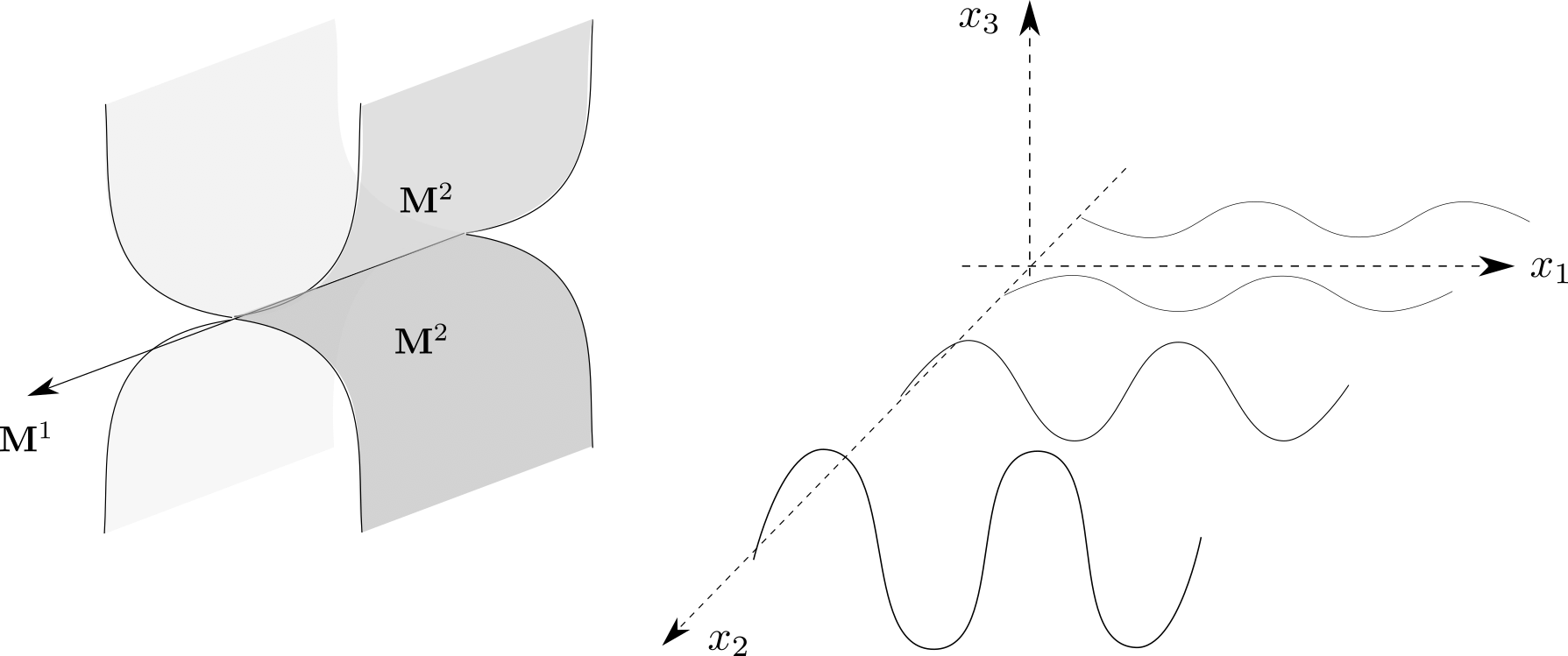}
   \caption{Left: a cusp. Right: a corrugated sheet.}
    \label{fig:tfs} 
   \end{center}
\end{figure}

\begin{remark}
    It may be thought that, using the fact that $\Man{k}$ is a $k$-dimensional submanifold, it can
    be flatten as in Definition~\ref{def:TFS} and maybe \HSTTFS-$(iii)$ could be always true.
    Unfortunately, this is not clear as shown by the example of a ``corrugated sheet''.
    Suppose that, after the flattening of $\Man{1}$, we end up with
    $$ \tMan{0}=\emptyset \quad,\quad \tMan{1}=\{(x_1,x_2,x_3):\,x_2=0,x_3=0\}\; ,$$
    $$\tMan{2}=\{(x_1,x_2,x_3):\,x_2\neq 0,x_3=x_2\sin(x_1)\}\;,$$
    and $\tMan{3}=\R^3\setminus(\Man{0}\cup \Man{1}\cup\Man{2})$. 
    In this situation, it is clear enough that $\tMan{2}$ does not satisfy \HSTTFS-$(iii)$. 

    The reader could argue that we may use an other change of variables in order to flatten
    $\tMan{2}$. This is probably right but this means that $(i)$ flattening $\Man{1}$ is not enough;
    $(ii)$ using an other change of variables to flatten $\Man{2}$ may be possible here, but more
    difficult and perhaps impossible if we consider an example where $\Man{2}$ has several connected
    components having $\Man{1}$ as boundary. We would face again the difficulty of ``simultaneous
    flattening''.
\end{remark}

Throughout the rest of the book, unless otherwise specified we will always assume that we are in the
framework of tangentially flattenable stratifications.

\section{Partial regularity, partial regularization}
\label{sect:sup.reg}

In this section, motivated by Sections~\ref{sect:htc} and~\ref{sect:whitney}, we present some key
ingredients in the proof of local comparison results for HJ Equations with discontinuities. The
assumptions we are going to use are those which are needed everywhere in this book to prove any kind
of results and therefore we define at the end of the section a ``good local framework for HJ
Equations with discontinuities''.

Local comparison results lead to consider HJ-Equations in a ball, namely
\begin{equation}\label{gen-eqn-F}
\G(\vc,u,Du)=0\quad \hbox{in  } B_\infty (\vcb,r)\; ,
\end{equation}
where $\vcb\in\R^{N}$ and $r>0$ are fixed. We recall that the notation $X$ can refer to either $X=x$
or $X=(x,t)$. Because of the previous section, it is natural to assume that the discontinuities in
this equation have a general \emph{\TFS stratification-type structure} and, near a point of
$\Man{k}$, after a suitable change of variables, we can assume that the variable $\vc \in \R^{N}$
can be decomposed as $(\vt,\vn)\in \R^k \times \R^{N-k}$ and $\G$ is
continuous \wrt $u$, $p$ and $\vt$ but not with respect to $\vn$. In particular we have in mind
that locally around $\vcb$, Hamiltonian $\G$ has a discontinuity on $\Gamma_0=\{(\vt,\vn);\ \vn=0\}$
which can be identified with $\R^{k}$.

The properties of discontinuous sub and supersolutions on $\Gamma$ are playing a key role in the
proof of such local comparison results and the aim of the next section is to introduce the notion of
``regular discontinuous function''.

\subsection{Regular discontinuous functions}\label{sec:regular}

The following definition provides several notions of regularity for discontinuous functions. 

\begin{definition}\label{def:regular}\index{Regular discontinuous function}
    \emph{--- Regular discontinuous functions.}\smsp
    Let $A\subset\R^k$,
    $f:A\to \R$ an \usc \resp{\lsc} function and $\omega\subset A$.
    \begin{enumerate}
        \item[$(i)$] The function $f$ is said to be $\omega$-regular at $x\in \partial \omega\cap A$
            if $$ f(x)=\limsup_{\substack{ y\to x \\ y\in \omega}}f(y) \qquad \bigl[resp.\
            f(x)=\liminf_{\substack{ y\to x \\ y\in \omega}}f(y)\bigr].$$
        \item[$(ii)$] Let $\mathcal{E}\subset \partial \omega \cap A$. The function $f$ is said to be
            $\omega$-regular on $\mathcal{E}$ if it is $\omega$-regular at any point of~$\mathcal{E}$.

        \item[$(iii)$] Let $\mathcal{E}\subset A$. Given $x\in\mathcal{E}$ and $r>0$, we denote by
            $\VV(x,r)$ the set of all connected components of $(A\setminus\mathcal{E})\cap B(x,r)$.
            We make the following assumption: for any $x\in\mathcal{E}$, there exists $r_0=r_0(x)>0$ such
            that
            \begin{equation}\label{hyp.comp}
            \begin{cases}
                \hbox{for all  } \omega\in\VV(x,r_0)\;,\ x\in\partial\omega\;,\\[2mm]
                \hbox{if  }0<r<r_0\;,\ \VV(x,r)=\Big\{ \omega\cap B(x,r)\;,\text{ where }
                \omega\in\VV(x,r_0)\Big\}\;.
            \end{cases}
            \end{equation}
            The function $f$ is said to be regular on $\mathcal{E}$ if, for any $x\in\mathcal{E}$
            and $0<r<r_0(x)$, $f$ is $\omega$-regular at $x$ for all $\omega\in\VV(x,r)$.
    \end{enumerate}
\end{definition}

Let first explain the admittedly strange assumption in \eqref{hyp.comp}. The first one is to
avoid pathological example like
$$ A=\bigcup_{n\geq 1}(\frac{1}{n+1},\frac{1}{n})\cup \{0\}\quad ,\quad \mathcal{E}=\{0\}\; .$$
Here $[A\setminus \mathcal{E}]\cap (-r,r)$ contains an infinite numbers of connected components
$\omega$ but none of them satisfies $0\in \overline{\omega}$. Clearly this is not the type of
situations we wish to handle and therefore the assumption excludes them.

The second assumption is to avoid appearance of vanishing of connected components as $r\to0$: this
assumptions means that the decomposition in connected components does not change for $r>0$ small.

On the other hand, we can consider $A=[-1,1]\times[-1,1]$ and $\mathcal{E}=[0,1]\times \{0\}$. It is
clear that $A\setminus \mathcal{E}$ is connected but we are interested in the {\em local situation},
not in the global one.  If $x=(x_1,0) \in \mathcal{E}$ is such that $0<x_1\leq 1$ then, for $r$
small enough, $[A\setminus \mathcal{E}]\cap B(x,r)$ has two connected components and ``regular'' at
such point for a  discontinuous function means regular ``from both sides'' of the segment
$\mathcal{E}$, \ie with respect to the two connected components.  Of course, if $x_1=0$, we come
back to the case when we only have one connected component.

In this book, this local aspect will always be important since almost all the arguments are local.
But concerning $A$ and $\mathcal{E}$, we will often be in a simple situation like $A=\R^N\times
(0,\Tf)$ and $\mathcal{E}=\M \times (0,\Tf)$ where $\M\subset \R^N$ is a $k$-dimensional manifold.
We point out anyway that here there are two different cases: if $k<N-1$, $\omega=A\setminus
\mathcal{E}$ is connected and there is no difference between $(ii)$ and $(iii)$. But if
$\mathcal{E}$ is an hyperplane, then, as in the above example, $A\setminus \mathcal{E}$ has two
connected components $\omega_1, \omega_2$ and, roughly speaking, the regularity property has to hold
in both side of $\mathcal{E}$, i.e. both for $\omega_1$ and $\omega_2$. This is actually the case
which will be studied in Part~\ref{part:codim1} and~\ref{part:NA}.

The regularity of \usc subsolution or \lsc supersolutions is used in several type of situations: the
most classical one is when we consider a stationary HJ Equation set in a domain $\Omega$ of $\R^N$;
a natural choice is $A=\Omegb$, $\omega=\Omega$, $\mathcal{E}=\domeg$. In the study of the Dirichlet
problem (\cf for example  \cite{BP1,BP2,BP3}), such regularity of the sub and/or supersolution is
needed to have a comparison result up to the boundary. The point is to avoid ``artificial values''
of these sub or supersolution on $\domeg$. For the case of evolution equations, one may also choose
$A=\Omegb\times (0,\Tf)$, $\omega=\Omega\times (0,\Tf)$, $\mathcal{E}=\domeg\times (0,\Tf)$.  In the same
context, some result can be  formulated using the $\omega$-regularity of the sub or  supersolution
at some point of $\domeg\times (0,\Tf)$, \cf Section~\ref{sect:eqn-on-b}.

In most of these applications, the assumption imposed on $\Omegb$ and $\domeg$ by
Definition~\ref{def:regular}-$(iii)$ is obviously satisfied but, if the domain is less regular,
typically as in the above example 
$$ \Omega=\left [ (-1,1)\times (-1,1) \right ] \setminus \left [ [0,1)\times \{0\}\right ]\; ,$$
then a more general notion of regularity can be useful. We refer to Part~\ref{S-BC} for a discussion
of such boundary regularity.

\subsection{Regularity of subsolutions}
\label{sec:regu-subsol}

The aim of this section is to study subsolutions of \eqref{gen-eqn-F} and to prove that, under
suitable assumptions, they satisfy some ``regularity properties''.

We immediately point out that, for reasons which will clear later on in this book, we are not going
to use only subsolutions in the Ishii sense and therefore, we are not going to use only the lower
semi-continuous enveloppe of some Hamiltonian as in the Ishii definition. To simplify matter, we
assume here that the function $\G$ contains all the necessary information for subsolutions. In
other words, by subsolution of \eqref{gen-eqn-F}, we mean an \usc function $u$ which satisfies

{\em\noindent At any 
maximum point $\vc \in B_\infty (\vcb,r)$ of $u-\phi$, where $\phi$ is a smooth test-function, we have
$$ \G(\vc,u(\vc), D\phi (\vc) )\leq 0 \; .$$
}
In the sequel, we decompose $Du$ as $(D_\vt u,D_\vn u)$ (the same convention is used for the
test-functions $\phi$) and the corresponding variable in $\G$ will be $p=(p_\vt,p_\vn)$.

In order to state our main result on the ``regularity of subsolutions'', let us introduce the assumption

\label{page:wNC}
\begin{assumption}{\NCw}{Weak Normal Controllability.}\\[-1.3cm]
    \begin{enumerate}
    \item[$(i)$] If $N-k>1$, there exists $e \in \R^{N-k}$ such that, for any $R>0$, we have
    $$
    \G(\vc,u,(p_\vt,Ce))\to +\infty \quad\hbox{when  }C \to +\infty\;,
    $$
    uniformly for $\vc=(\vt,\vn)\in B_\infty (\vcb,r)$, $|u|\leq R$, $|p_\vt|\leq R$. 
    \item[$(ii)$] If $N-k=1$, this property holds for $e =+1$.
    \item[$(iii)$] If $N-k=1$, this property holds for $e =-1$.
    \end{enumerate}
\end{assumption}

The results are the following:\index{Regularity of subsolutions!general result} 
\begin{proposition}\label{reg-sub} We consider equation~\eqref{gen-eqn-F} in $B_\infty(\bar X,r)$.
\begin{enumerate}
\item[$(a)$] Assume that \NCw holds. If $u$ be a bounded, \usc subsolution of \eqref{gen-eqn-F} and if 
$\Gamma_c:=B_\infty (\vcb,r)\cap \{(\vt,\vn)\;;\; \vn=c\}\neq
\emptyset$, then $u$ is regular on $\Gamma_c$\;. In particular, $u$ is regular on $\Gamma=\Gamma_0$\;.
\item[$(b)$] If $u$ be a bounded, \usc subsolution of \eqref{gen-eqn-F}, if $N-k=1$ and if \NCw-$(ii)$
    holds, then $u$ is regular on $\Gamma_0$ with respect to $B_\infty (\vcb,r)\cap\{Z>0\}$. In the
    same way, if $N-k=1$ and if \NCw-$(iii)$ holds, then $u$ is regular on $\Gamma_0$ with respect
    to $B_\infty (\vcb,r)\cap\{Z<0\}$.
\item[$(c)$] If $u$ is a subsolution of $\G=0$ on $B_\infty (\vcb,r)\cap\{Z\geq0\}$ and if either
    \NCw-$(ii)$ or \NCw-$(iii)$ holds then $u$ is is regular on $\Gamma_0$ with respect to
    $B_\infty (\vcb,r)\cap\{Z>0\}$.
\end{enumerate}
\end{proposition}

This proposition means that in $B_\infty (\vcb,r)$, subsolutions cannot have ``singular values'' on
affine subspaces of the form $\{\,(\vt,\vn);\ \vn=c\,\}$. By singular values we mean here values
which are not given by limits coming from outside of those affine subspaces. The three above results
can be interpreted in the following way: (a) is the general ``good case'' of a subsolution which is
regular on $\Gamma_0$, a set of discontinuity for $\G$, when we use the entire assumption \NCw.
Result (b) is the case when $\Gamma_0$ is an affine hyperplan but only one part of assumption \NCw.
Result (c) deals with boundary regularity; such regularity property is useful in order to use the
results of Section~\ref{sect:eqn-on-b}.

\begin{proof}
    We start by (a). We recall that, thanks to Definition~\ref{def:regular}, in the case when $k<N-1$, we
    have to show that, for any $\vc=(\vt,\vn)\in \Gamma_c$
\begin{equation}\label{reg-kp}
u(\vc)=\limsup\{u(\vt',\vn')\ ; (\vt',\vn')\to \vc,\ \vn'\neq \vn\}\; .
\end{equation}
since $B_\infty (\vcb,r)\setminus \Gamma_c$ is connected and $\vn'\neq \vn$ is equivalent to
$(\vt',\vn')\notin \Gamma_c$.  Moreover, if $N-k=1$, we also have to show
\begin{align}
u(\vc)&=\limsup\{u(\vt',\vn'); (\vt',\vn')\to \vc,\ \vn'> \vn\} \nonumber \\
&=\limsup\{u(\vt',\vn'); (\vt',\vn')\to \vc,\ \vn'< \vn\} ,\label{reg-kp-codim1}
\end{align}
since in this case, $B_\infty (\vcb,r)\setminus \Gamma_c$ has two connected components. 
In order to prove \eqref{reg-kp} we argue by contradiction assuming that
$$ u(\vc)>\limsup\{u(\vt',\vn')\ ; (\vt',\vn')\to \vc,\ \vn'\neq \vn\}\; .$$
Therefore there exists some $\delta>0$ small enough such that $u(\vt',\vn')<u(\vc) -\delta$ if
$|(\vt',\vn')-\vc| <\delta$, with $\vn'\neq \vn$.  Next, for $\e >0$, we consider the function
$$\vt' \mapsto u(\vt',\vn) - \frac{|\vt-\vt'|^2}{\eps}\; .$$
If $\e $ is small enough, this function has a local maximum point at $\vte$ which satisfies $|\vte
    -\vt|<\delta$ and $u(\vte,\vn)\geq u(\vc)$. But because of the above property, there exists a
    neighborhood $\VV$ of $(\vte,\vn)$ such that, if $(\vt',\vn') \in \VV$ and $\vn'\neq \vn$,
    $u(\vt',\vn')<u(\vte,\vn) -\delta$.

This implies that $(\vte,\vn)$ is also a local maximum point of the function
$$(\vt',\vn') \mapsto u(\vt',\vn') - \frac{|\vt-\vt'|^2}{\eps}-Ce\cdot (\vn'-\vn)\; .$$
for any positive constant $C$ and the vector $e$ of $\R^{N-k}$ given by \NCw. But, by the subsolution
    property, we have $$\G\left((\vte,\vn),u(\vte,\vn ),
    \left(\frac{2(\vte-\vt)}{\eps},Ce\right)\right)\leq 0\; .$$
But, using \NCw with $R=\max(||u||_\infty,2\delta\e^{-1})$, we reach a contradiction for $C$ large enough.

For the case $N-k=1$, we repeat the same argument by choosing either $e=+1$ or $e=-1$.

Indeed, if we assume by contradiction that $u(\vc)>\limsup\{u(\vt',\vn')\ ; (\vt',\vn')\to \vc,\ \vn'> \vn\}$, we argue
as above but looking at a local maximum point of the function 
$$(\vt',\vn') \mapsto u(\vt',\vn') - \frac{|\vt-\vt'|^2}{\eps}+C (\vn'-\vn)\; ,$$
therefore with the choice $e=-1$.  We first look at a maximum point of this function in compact set
of the form $$ \{(\vt',\vn'); |\vt'-\vt|+|\vn'-\vn|\leq \delta,\ \vn'\leq \vn\}\; .$$ 

Notice that, in this set, the term $C\cdot (\vn'-\vn)$ is negative (therefore it has the right
sign) and this function has a local maximum point which depends on $\e$ and $C$, but, in order to
simplify the notations, we denote it by $(\bar \vt, \bar \vn)$. We have
$u(\bar \vt, \bar \vn) \geq u(\vc)$ by the maximum point property and we have $(\bar \vt, \bar \vn)\to (\vte,0)$ when
$C\to +\infty$, where $(\vte,0)$ is a maximum point of the function
$$\vt' \mapsto u(\vt',0) - \frac{|\vt-\vt'|^2}{\eps}\; .$$

Using that $u(\vc)>\limsup\{u(\vt',\vn')\ ; (\vt',\vn')\to \vc,\ \vn'> \vn\}$, we clearly have the
same property at $(\vte,0)$ and therefore, for $C$ large enough, at $(\bar \vt, \bar \vn)$ which
is also a maximum point of the above function for all $(\vt',\vn')$ such that $|\vt'-\vt|+|\vn'-\vn|\leq \delta$ if $\delta$ is chosen small enough. And we reach a contradiction as in the first part of the proof using \NCw.

Hence $u$ is regular with respect to the the $\{\vn'> \vn\}$ side but an analogous proof shows the same property for
the other side.

Finally the proofs of (b) and (c) rely on analogous arguments, therefore we skip them. We just point out that, for (c),
the fact that $B_\infty (\vcb,r)\cap\{Z<0\}$ is not part of the domain allows to do the proof as in the first case of (a).
\end{proof}

\begin{remark}\label{rem:reg-sub-weak} \ 
    \begin{enumerate}
        \item[$(i)$] We have stated and proved Proposition~\ref{reg-sub} under Assumption \NCw but,
            in the sequel, we will mainly use Assumption \NCe which will be introduced in the next
            section.  Clearly \NCe implies \NCw.

            On an other hand, we point out that, in control problems, provided that the Hamiltonian
            $\G$ is defined in a suitable way, \NCw is equivalent to the existence of a
            non-tangential dynamic in the case $N-k>1$ while, in the case when $N-k=1$, it is
            equivalent to the  existence of two dynamics pointing strictly inward each of the two
            half-spaces defined by the hyperplan $\Gamma_0$.

        \item[$(ii)$] Notice that a similar result still holds for \lsc subsolutions à la
            Barron-Jensen, where we consider minimum points of $u-\phi$. Of course in this case, the
            regularity property  has to be expressed with a liminf instead of a limsup but the
            modifications are straightforward. We refer to Section~\ref{sect:SBJ} where the
            Barron-Jensen approach is detailed and we use this liminf regularity property.
    \end{enumerate}
\end{remark}

\subsection{Regularization of subsolutions}
\label{sec:regplus-subsol}

The aim of this section is to construct, for a given subsolution, a suitable approximation by Lipschitz
continuous subsolutions which are even $C^1$ in $\vt$ in the convex case.

To do so, we use for $\G$ the following assumptions: for any $R>0$, there exist some constants $C^R_i>0$ for
$i=1\dots 4$, a modulus of continuity $m^R : [0,+\infty[ \to [0,+\infty[$ and either a constant
$\lambda^R>0$ or $\mu^R>0$  such that

\index{Hamiltonians!tangential continuity}\index{Tangential continuity!Hamiltonian version}\label{page:TC}\label{not:TC}
\begin{assumption}{\TC}{Tangential Continuity.}
    For any $\vc_1=(\vt_1,\vn), \vc_2=(\vt_2,\vn) \in B_\infty (\vcb,r)$, $|u|\leq R$ and $p\in \R^{N}$,
    then 
    $$|\G(\vc_1,u,p)- \G(\vc_2,u,p)|\leq C^R_1|\vt_1-\vt_2|.|p| +m^R\big(|\vt_1-\vt_2|\big)\;.$$
\end{assumption}

\vspace*{-1.5em}

\index{Hamiltonians!normal controllability}\index{Normal controllability!Hamiltonian version}\label{page:NC}\label{not:NC}
\begin{assumption}{\NCe}{Normal Controllability.}
    For any $\vc=(\vt,\vn)\in B_\infty (\vcb,r)$, $|u|\leq R$, $p=(p_\vt,p_\vn)\in \R^{N}$, then 
    $$\G(\vc,u,p)\geq C^R_2 |p_\vn| - C^R_3|p_\vt| -C^R_4\;.$$
\end{assumption}

\vspace*{-0.7em}

Notice that \NCe and \TC have counterparts in terms of control elements \ie dynamic and cost, see
\NCBCL, \TCBCL, p. \pageref{page:TCBCL}.  For the last assumption, if $p_\vt\in \R^k$, we set
$p_\vt=(p_{\vt_1},\cdots,p_{\vt_k})$ 
\index{Hamiltonians!monotonicity}\label{page:Mong}\label{not:Mong}

\begin{assumption}{\Mong}{Monotonicity.}
    For any $R>0$, there exists $\lambda_R, \mu_R \in \R$, such that one of the two following properties holds
\begin{enumerate}
    \item[] \Monu : $\lambda_R>0$ and for any $\vc\in B_\infty (\vcb,r)$, $p=(p_\vt,p_\vn)\in \R^{N}$,
        any $- R\leq u_1 \leq u_2 \leq R$, 
        \begin{equation}\label{mon-u}
        \G(\vc,u_2,p)-\G(\vc,u_1,p)\geq \lambda^R(u_2-u_1)\;; 
        \end{equation}
    \item[] \Monp :
         \eqref{mon-u} holds with $\lambda_R=0$, we have $\mu_R>0$ and
        \begin{equation}\label{mon-p}
            \G(\vc,u_1,q)-\G(\vc,u_1,p)\geq \mu^R(q_{\vt_1}-p_{\vt_1})\; ,
        \end{equation}
        for any $q=(q_\vt,p_\vn)$ with $p_{\vt_1}\leq q_{\vt_1}$ and
        $p_{\vt_i}=q_{\vt_i}$ for $i=2,...,p$.
\end{enumerate}
\end{assumption}

Before providing results using these assumptions, we give an example showing the type of properties
hidden behind these general assumptions. 

\

\begin{example}
We consider an equation in $\R^{N+1}$ written as 
$$ \mu u_t + H((x_1, x_2), t, u, (D_{x_1} u, D_{x_2} u))= 0\quad 
\hbox{in  }\R^k \times \R^{N-k} \times (0,+\infty)\; ,$$
Here the constant $\mu$ satisfies $0\leq \mu \leq 1$ and in order to simplify we can assume that $H$
is a continuous function. To be in the above framework, we write $\vc=(t,x_1, x_2) \in
(0,+\infty)\times \R^k \times \R^{N-k}$ and we set $\vt=(t,x_1) \in \R^{k+1}$, $\vn=x_2\in R^{N-k}$
and $$ \G(\vc , u, P) = \mu p_{t} + H((x_1,x_2),t,u,(p_{x_1},p_{x_2}))\; ,$$ where
$P=(p_t,(p_{x_1},p_{x_2}))$. 

In order to formulate \TC, \NCe and \Mong in a simple way, we assume that $(x_1,t,u) \mapsto
H((x_1,x_2),t,u,(p_{x_1},p_{x_2}))$ is locally Lipschitz continuous for any $x_2, p_{x_1},p_{x_2}$.
Then these assumptions can be formulated in the following way

\noindent\textbullet\; For \TC, recalling that we always argue {\em locally}, one has to assume
that, for any $R>0$, there exists a constant $C^R_1>0$ such that, for any $(t,x_1, x_2) \in
[0,+\infty)\times \R^k \times \R^{N-k}$ with $t+|x_1|+|x_2|\leq R$, $|u|\leq R$ and
$(p_{x_1},p_{x_2})\in  \R^k \times \R^{N-k}$, we have
$$ |D_{x_1} H((x_1,x_2),t,u,(p_{x_1},p_{x_2}))|, |D_{t} H((x_1,x_2),t,u,(p_{x_1},p_{x_2}))| \leq
C^R_1(|(p_{x_1},p_{x_2})|+1)\; .$$
Here we are in the simple case when $m^R\big(\tau \big)= C^R_1 \tau$ for any $\tau
\geq 0$. One can easily check that these assumptions imply the right property for $\G$ with
$\vt=(t,x_1)$.

\noindent\textbullet\; Next since $p_{\vt_1}=p_t$, \Mong reduces to either $\mu >0$ or $D_{u}
H((x_1,x_2),t,u,(p_{x_1},p_{x_2})) \geq \lambda_R >0$ for the same set of
$(t,x_1,x_2),u,(p_{x_1},p_{x_2})$ as for \TC. Hence, either we are in a real time evolution
context ($\mu >0$), or $\mu = 0$ and the standard assumption ``$H$ strictly
increasing in $u$'' has to hold.

\noindent\textbullet\; Finally \NCe holds if $H$ satisfies the following coercivity assumption in $p_{x_2}$
$$H((x_1,x_2),t,u,(p_{x_1},p_{x_2})) \geq C^R_2 |p_{x_2}| - C^R_3|p_{x_1}| -C^R_4\;,$$ again for the
same set of $(t,x_1,x_2),u,(p_{x_1},p_{x_2})$ as for \TC. Notice that in order to check \NCe for
$\G$, the constant $C^R_3$ may have to be changed in order to incorporate the $\mu p_t$-term if $\mu
\neq 0$. 
\end{example}

Our result concerning the approximation by Lipschitz subsolutions is the\index{Regularization!subsolutions}
\begin{proposition}\label{reg-by-sc}\emph{--- Regularization of subsolutions.}\smsp 
    Let $u$ be a bounded subsolution of \eqref{gen-eqn-F} and
    assume that \TC, \NCe and \Mong hold. Then there exists a sequence of Lipschitz continuous
    functions $(u^\e)_\e$ defined in $B_\infty (\vcb,r-a(\e))$ where $a(\e) \to 0$ as $\e\to 0$ such
    that \\[-5mm]
    \begin{enumerate} 
        \item[$(i)$] each $u^\e$ is a subsolution of \eqref{gen-eqn-F} in $B_\infty (\vcb,r-a(\e))$,
        \item[$(ii)$] each $u^\e$ are semi-convex in the $\vt$-variable
        \item[$(iii)$] $\limssup u^\e = u$ as $\e\to 0$.
    \end{enumerate}
\end{proposition}

\begin{remark}Equations of the form
$$ \max(u_t + G_1(x,D_x u); G_2(x,u,D_x u)) = 0 \; ,$$
do not satisfy \Mong even if $G_2$ satisfies \Monu and the Hamiltonian $p_t + G_1(x,p_x)$ satisfies
    \Monp.  To overcome this difficulty, we have to use a change of variable of the form
    $v=\exp(Kt)\cdot u$ in order that both Hamiltonians satisfy \Monu, which is a natural change (cf.
    Section~\ref{sec:GFHJD}). Of course, suitable assumptions on $G_1$ and $G_2$ are needed in order
    to have \TC and \NCe.
\end{remark}

\begin{proof}
    First we can drop the $R$ in all the constants appearing in the assumptions by remarking that,
    $u$ being bounded, we can use the constants with $R=||u||_\infty$.

In the case, when \Mong holds with $\lambda>0$ we set for $\vc=(\vt,\vn)$
$$
u^{\e}(\vc):=\max_{\vt' \in \R^k}\Big\{u(\vt',\vn)-
    \frac{\left(|\vt-\vt'|^2+\e^4\right)^{\alpha/2}}{\e^\alpha}\Big\},
$$
for some (small) $\alpha >0$ to be chosen later on,
while, in the other case we set
$$
u^{\e}(\vc):=\max_{\vt' \in \R^k}
\Big\{u(\vt',\vn)-\exp(K\vt_1) \frac{|\vt-\vt'|^2}{\e^2}\Big\},
$$
for some constant $K$ to be chosen later on.

In both cases, the maximum is achieved for some $\vt'$ such that $ |\vt-\vt'|\leq O(\e)$, hence with a point $(\vt',\vn)\in B_\infty (\vcb,r)$ for $a(\e)>O(\e)$, and therefore $u^\e$ is well-defined in $B_\infty (\vcb,r-a(\e))$. By standard properties of the sup-convolution, the $u^{\e}$'s are continuous in $\vt$ but, for the time being, not necessarily in $\vn$, despite of Proposition~\ref{reg-sub}.

To prove that $u^\e$ is a subsolution in $B_\infty (\vcb,r-a(\e))$, we consider a smooth test-function $\phi$ and we assume that $\vc\in B_\infty (\vcb,r-a(\e))$ is a maximum point of $u^{\e} - \phi$. We first consider the ``$\lambda>0$'' case : if 
$$u^{\e}(\vc)=u(\vt',\vn)-
    \frac{\left(|\vt-\vt'|^2+\e^4\right)^{\alpha/2}}{\e^\alpha},$$
then $(\vt',\vn)$ is a maximum point of 
$(\tilde \vt,\tilde \vn) \mapsto u(\tilde \vt,\tilde \vn)-
    \e^{-\alpha}\big(|\vt-\tilde \vt|^2+\e^4\big)^{\alpha/2}-\phi(\vt,\tilde \vn)\; ,$
and therefore, by the subsolution property for $u$
$$ \G((\vt',\vn), u(\vt',\vn), (p_\vt,D_\vn \phi(\vt,\vn)))\leq 0\; ;$$
where 
$$p_\vt := \alpha (\vt'-\vt) \frac{\left(|\vt-\vt'|^2+\e^4\right)^{\alpha/2-1}}{\e^\alpha}\; .$$
On the other hand the maximum point property in $\vt$, implies that $p_\vt=D_\vt\phi(\vt,\vn)$.

To obtain the right inequality, we have to replace $(\vt',\vn)$ by $\vc=(\vt,\vn)$ in this inequality and $u(\vt',\vn)$ by $u^\e (\vc)$. To do so, we have to use \TC; in order to do it, we need to have a precise estimate on the term $|\vt-\vt'||(p_\vt,D_\vn \phi(\vt,\vn))|$.
The explicit form of $p_\vt$ gives it for $|\vt-\vt'||p_\vt|$ but this is not the case for $|\vt-\vt'|.|D_\vn\phi(\vt,\vn)|$ since we have not such a precise information on $D_\vn\phi(\vt,\vn)$. Instead we have to use \NCe which implies
$$ 
C_2 |D_\vn\phi(\vt,\vn)| - C_3|p_\vt| -C_4\leq 0\; .$$
(remember that we have dropped the dependence in $R$ for all the constants). On the other hand, we have combining \TC and \Mong
$$
\G(\vc, u^\e (\vc), (D_\vt\phi(\vt,\vn),D_\vn \phi(\vt,\vn))) \leq  \G((\vt',\vn), u(\vt',\vn), (p_\vt,D_\vn \phi(\vt,\vn))) + $$
$$ \ \qquad\qquad C_1|\vt-\vt'||D\phi (\vc)| +m(|\vt-\vt'|)-\lambda \frac{\left(|\vt-\tilde \vt|^2+\e^4\right)^{\alpha/2}}{\e^\alpha}\; .
$$
It remains to estimate the right-hand side of this inequality: we have seen above that $|\vt-\vt'|=O(\e)$ and \NCe implies that
$$ |D\phi (\vc)|\leq \bar K(|p_\vt|+1)\; ,$$
for some large constant $\bar K$ depending only on $C_2,C_3,C_4$. Finally
$$|\vt-\vt'||p_\vt|=\alpha |\vt-\vt'|^2 \frac{\left(|\vt-\vt'|^2+\e^4\right)^{\alpha/2-1}}{\e^\alpha}\leq \alpha \frac{\left(|\vt-\tilde \vt|^2+\e^4\right)^{\alpha/2}}{\e^\alpha}\; .$$
By taking $\alpha < \bar K$, we finally conclude that
$$
\G(\vc, u^\e (\vc), (D_\vt\phi(\vt,\vn),D_\vn \phi(\vt,\vn))\leq O(\e)+m(O(\e))\; ,$$
and changing $u^\e$ in $u^\e -\lambda^{-1}(O(\e)+m(O(\e)))$, we have the desired property.

In the $\mu$-case, the equality $p_\vt=D_\vt\phi(\vt,\vn)$ is replaced by
$$ D_\vt\phi(\vt,\vn)= -K\exp(K\vt_1) \frac{|\vt-\vt'|^2}{\e^2}e_1+ \exp(Kt) \frac{(\vt'-\vt)}{\e^2}\; ,$$
where $e_1$ is the vector $(1,0,\cdots,0)$ in $\R^{k}$. The viscosity subsolution inequality for $u$ at $(\vt',\vn)$ reads
$$ \G((\vt',\vn), u(\vt',\vn), (\tilde p_\vt,D_\vn \phi(\vt,\vn))\leq 0\; ,$$
where $\displaystyle \tilde p_\vt =\exp(Kt) \frac{(\vt'-\vt)}{\e^2}$. 

We first use \NCe, which implies 
$$ |D\phi (\vc)|\leq \bar K(|\tilde p_\vt|+1)=\bar K(\exp(Kt) \frac{|\vt'-\vt|}{\e^2}+1)\; .$$
Then we combine \TC and \Mong to obtain
$$
\G(\vc, u^\e (\vc), (D_\vt\phi(\vt,\vn),D_\vn \phi(\vt,\vn)) \leq  \G((\vt',\vn), u(\vt',\vn), (\tilde p_\vt,D_\vn \phi(\vt,\vn)) + $$
$$ \ \qquad\qquad C_1|\vt-\vt'||D\phi (\vc)| +m(|\vt-\vt'|)-\mu K\exp(K\vt_1) \frac{|\vt-\vt'|^2}{\e^2}\; .
$$
We conclude easily as in the first case choosing $K$ such that $\mu K>C_1\bar K$.

Properties $(ii)$ and $(iii)$ are classical properties which are easy to obtain and we drop the proof. 

We conclude this proof by sketching the proof of the Lipschitz continuity of $u^{\e}$ in $\vn$. To do so, we write $\vcb=(\vtb,\vnb)$ and for any fixed $\vt$ such that $|\vt-\vtb| < r-a(\e)$, we consider the function $\vn\mapsto u^{\e}(\vt,\vn)$. By using \NCe and the Lipschitz continuity of $u^{\e}$ in the $\vt$-variable, it is easy to prove that this function is a subsolution of 
$$
C_2 |D_\vn w| \leq  C_3K_\eps +C_4\; ,$$
where $K_\eps=||D_\vt u^{\e}||_\infty$ and the estimates of $D_\vn u^{\e}$ follows.
\end{proof}

\

\noindent\textbf{The convex case --}
The above regularization result can be improved when some convexity property of the Hamiltonian holds. 
More precisely, let us introduce the following assumption

\begin{assumption}{\HConv}{Convexity assumption.}
    \label{page:HConv} 
    For any $\vc\in B_\infty (\vcb,r)$, the function $(u,p)\mapsto \G(\vc,u,p)$ is convex.
\end{assumption}

We begin with a result concerning convex combinations of subsolutions. While the result is
interesting in itself even in the case of continuous Hamiltonians, we actually need it to make a
suitable regularization of subsolutions. By a convex combination of subsolutions $u_i$ for $i=1,..,n$,
we mean of course a finite sum
$$W:=\sum_{i=1}^n \mu_iu_i\;,\quad\text{where for all  } i,\ 
\mu_i\geq0\quad \hbox{and}\quad\sum_{i=1}^n\mu_i=1\;. $$

\begin{lemma}\label{combconvsub} Assume that $(X,r,p)\mapsto\G(X,r,p)$ is \lsc and 
satisfies \HConv. Then any convex combination of Lipschitz continuous subsolutions of $\G=0$ is a subsolution 
of $\G=0$. 
\end{lemma}

\begin{proof} We just sketch the proof since most of the arguments are rather standard.
We have only to prove the result for a convex combination of two subsolutions
$W:=\lambda w_1+(1-\lambda)w_2$, the general case involving $n$ subsolutions for $n>2$ deriving
immediately by iteration of the result. Of course, we can assume w.l.o.g. that $0<\lambda<1$.

Let $\phi$ be a smooth test-function and $\tX \in B_\infty (\vcb,r)$ a local strict maximum point of
$W-\phi$ in $\overline{B (\tX,\tr)} \subset B_\infty (\vcb,r)$. We use a tripling of variables
by considering in $\overline{B (\tX,\tr)}^3$ the function
$$  \Psi(X_1,X_2,X):=\lambda w_1(X_1)+(1-\lambda)w_2(X_2)-\phi(X)-\lambda\frac{|X_1-X|^2}{\e}-
    (1-\lambda)\frac{|X_2-X|^2}{\e}\;.$$ 
Denoting by $(X_1^\e,X_2^\e,X^\e)$ a maximum point of this function and applying
Lemma~\ref{lem:cv-pen} in Section~\ref{sec:cv-pen}, we have $(X_1^\e,X_2^\e,X^\e) \to
(\tX,\tX,\tX)$ when $\e \to 0$, therefore $X_1^\e,X_2^\e,X^\e \in B (\tX,\tr)$ for $\e$ small
enough. Hence we get the viscosity inequalities
$$  \G(X_1^\e,w_1(X_1^\e),P_1^\e)\leq 0\quad , \quad \G(X_2^\e,w_2(X_2^\e),P_2^\e)\leq 0\; ,$$
and the property $D\phi(X)=\lambda P_1^\e+ (1-\lambda) P_2^\e$, where, for $i=1,2$,
$$P_i^\e=\frac{2(X^\e_i-X^\e)}{\e}\;. $$
Using the Lipschitz continuity of $w_1,w_2$, the $P_i^\e$ are uniformly bounded with respect to
$\e$ and extracting if necessary subsequences, we can
assume that they converge respectively to $P_i$ when $\e\to 0$.

Letting $\e$ tend to $0$, using in a crucial way the lower semi-continuity of $\G$, we are
lead to the same situation as above:
$$  \G(\tX,w_1(\tX),P_1)\leq 0\quad ,\quad  \G(\tX,w_2(\tX),P_2)\leq 0\;.$$
Because of the continuity of $D\phi$, $D\phi(\tX)=\lambda P_1+ (1-\lambda) P_2$. So, making the
convex combinaison of the above inequalities, after using \HConv we finally get
$$\G(\tX,W(\tX),D\phi(\tX))\leq 0\;,$$
which proves that $W$ is a viscosity subsolution of $\G=0$.    
\end{proof}

We can now state the regularization result
\index{Regularization!subsolutions, convex case}

\begin{proposition}\label{C1-reg-by-sc}\emph{--- Regularization of subsolutions, convex case.}\smsp 
    Under the assumptions of Proposition~\ref{reg-by-sc},  if $\G$ is \lsc and \HConv holds,
    the sequence $(u^\e)_\e$ of Lipschitz continuous subsolutions of \eqref{gen-eqn-F} can be built
    in such a way that they are $C^1$ (and even $C^\infty$) in the $\vt$ variable.
\end{proposition}

\begin{proof} By Proposition~\ref{reg-by-sc}, we can assume without loss of generality that $u$ is
    Lipschitz continuous. In order to obtain further regularity, we are going to use a standard
    convolution with a sequence of mollifying kernels but only in the $\vt$-variable.

    Let us introduce a sequence $(\rho_\eps)_\eps$ of positive, $C^\infty$-functions on $\R^k$,
    $\rho_\eps$ having a compact support in $B_\infty(0,\eps)$ and with
    $\int_{\R^k}\,\rho_\eps(e)de=1$. Then we set, for $\vc=(\vt,\vn) \in B_\infty (\vcb,r - \eps)$
    $$u^\e(\vc):= \int_{|e|_\infty <\eps} u(\vt-e,\vn)\rho_\eps(e)de\; .$$
    By standard arguments, it is clear that $u^\e$ is smooth in $\vt$.

    We first want to prove that the $u^\e$ are approximate subsolutions of
    \eqref{gen-eqn-F}, i.e. there exists some $\eta(\e) \to 0$ as $\e\to0$ such that 
    \begin{equation}\label{def:approx.sub}
        \G(\vc,u^\e,Du^\e)\leq \eta (\e) \quad \hbox{in  }B_\infty(\vcb,r-\eps)\;.
    \end{equation}
    To do so we follow the strategy of \cite{BJ:rate}[Lemma A.3], approximating the integral by a
    Riemann sum. We are lead to consider a function $u_n^\e$ defined by 
    $$u_n^\e(X):=\sum_{i=1}^n \mu_i\, u(Y-e_i,Z)\;,$$ 
    for some $|e_i|<\e$ and for coefficients $\mu_i\geq0$ such that $\sum \mu_i=\int \rho_\e=1$.

    Using \TC and the Lipschitz continuity of $u$, it is clear that there exists $\eta (\e)$, satisfying
    the above mentioned properties and independent of $i$, such that the functions $X=(Y,Z)\mapsto
    u(Y-e_i,Z)$ are all subsolutions of $\G-\eta(\e)=0$.
    
    Applying Lemma~\ref{combconvsub}, $u_n^\e$ is also a subsolution of $\G-\eta(\e)=0$ and since
    $u_n^\e$ converges uniformly to $u^\e$ when $n\to +\infty$, a standard stability result (cf.
    Theorem~\ref{hrl}) implies that $u^\e$ a subsolution of $\G-\eta(\e)=0$ as well.

    Finally, in order to drop the $\eta (\e)$-term in \eqref{def:approx.sub},
    we can either replace $u_\e$ by $u^\e-\lambda^{-1}\eta (\e)$ in the \Monu-case, 
    or $u^\e-\mu^{-1}\eta (\e)\vt_1$ in the \Monp case, and we get indeed a subsolution
    of $\G=0$. 
\end{proof}

\begin{remark}\label{rem:r-nonconvex} Let us make three complementary comments.\\[2mm]
     \noindent $(i)$ It is clear from the proof of Proposition~\ref{C1-reg-by-sc} that the convexity
            of $\G(X,r,P)$ in $r$ is not necessary to obtain such a result, the continuity in
            $r$ being enough, as we explain now. Notice first that, by the Lipschitz continuity of $u$,
            $$|u(Y-e_i,Z)-u(Y,Z)|, |u(Y,Z)-u^\e(Y,Z)|, |u(Y,Z)-u^\e_n(Y,Z)|\leq K\e\;,$$ 
            $K$ being the Lipschitz constant. Then, we are reduced to a version of
            Lemma~\ref{combconvsub} with no $r$-dependence by using an approximate Hamiltonian of
            the form $\tilde G(X,P):=\G(X,u^\e(X),P)-\tilde \eta (\e)$, depending only on $X$ and $P$.
            Indeed, taking into account the $K\e$ error term into $\tilde\eta (\e)$, the functions
            $u(Y-e_i,Z)$, $u_n^\e$ and $u^\e$ become all subsolutions of $\tilde G=0$, which satisfies
            \HConv. The rest of the proof is then the same as above.
        
        \noindent $(ii)$ The next remark concerns ``tangential regularizations'' in the case of a
            ``tangential viscosity inequalities''. In several situations, and in particular in
            stratified problems, the subsolution $u$ of \eqref{gen-eqn-F} satisfies also a
            subsolution inequality of the form $$ \G^\Gamma(\vt,u(\vt,0),D_\vt u(\vt,0))\leq 0 \quad
            \hbox{on  }\Gamma \; ,$$ where the precise meaning of this subsolution inequality is
            obtained by looking at maximum points of $u(\vt,0)-\phi(\vt)$ on $\Gamma$, not in all
            $\R^N$. As the proofs of Proposition~\ref{reg-by-sc} and \ref{C1-reg-by-sc} show, if
            $\G^\Gamma$ satisfies \TC, then the $u^\e$ given by the regularization processes of
            these results are also semi-convex or $C^1$ subsolutions of $\G^\Gamma\leq 0$; indeed
            the main difficulty in the proofs of these results comes from the $Z$-variable which
            does not appear here. A remark which plays a crucial in the case of stratified problems.

        \noindent $(iii)$ The result still applies to quasi-convex Hamiltonians. Indeed, for
        instance using $(i)$ above for simplicity, the convexity of $\G$ is used to prove
        essentially that if $G(X,u,P_i)\leq0$ for $i=1,2$, then $\G(X,u,sP_1+(1-s)P_2)\leq0$. But of
        course this is also true in the case of quasi-convexity since
        $$\G(X,u,sP_1+(1-s)P_2)\leq\max\{\G(X,u,P_1),\G(X,u,P_2)\}\leq0\;.$$
        However, in the context of evolution equations this means that we need a ``full''
        quasi-convexity assumption: under the form $H(x,t,u,(D_xu,u_t))=0$, the quasi-convexity is
        required to hold with respect to both $(D_xu,u_t)$. Suprisingly, this assumption leaves out
        ``natural'' evolution equations under the form $u_t+F(x,t,u,D_xu)=0$ where $F$ is quasi-convex
        in $D_xu$. Indeed, the full Hamiltonian $H=u_t+F$ is not quasi-convex with respect to both
        variables in general.  
\end{remark}

\subsection{What about regularization for supersolutions?}\label{sec:regul-supersol}

\index{Regularization!supersolutions}
The previous section shows how to regularize subsolutions and we address here the question: is it
possible to do it for supersolutions, changing (of course) the sup-convolution into an
inf-convolution? 

Looking at the proof of Theorem~\ref{reg-by-sc}, the answer is not completely obvious: on one hand,
the arguments for an inf-convolution may appear as being analogous but, on the other hand, we use in
a key way Assumption \NCe which allows to control the derivatives in $\vn$ of the sup-convolution (or
the test-function), {\em an argument which is, of course, valid only for subsolutions}.

Actually, regularizing a supersolution $v$ of \eqref{gen-eqn-F}---a notion which is defined exactly
in the same way as for subsolutions---requires additional assumptions on either $v$ or $\G$. For
$\G$, we introduce the following stronger version of \TC

\index{Tangential continuity!strong}\index{Hamiltonians!strong tangential continuity}
\begin{assumption}{\TCs}{Strong Tangential Continuity.}
    \label{page:TCs} For any $R>0$, there exists $C^R_1>0$ and a modulus of continuity $m^R :
    [0,+\infty[ \to [0,+\infty[$ such that for any $\vc_1=(\vt_1,\vn), \vc_2=(\vt_2,\vn) \in
    B_\infty (\vcb,r)$, $|u|\leq R$, $p=(p_\vt,p_\vn)\in \R^{N}$, then 
    $$|\G(\vc_1,u,p)- \G(\vc_2,u,p)|\leq C^R_1|\vt_1-\vt_2|.|p_\vt| +m^R\big(|\vt_1-\vt_2|\big)\;.$$
\end{assumption}
\vspace*{-0.5em}

We point out that, compared to \TC, the ``$|p|$'' is replaced by ``$|p_\vt|$''. This assumption is
typically satisfied by equations of the form $$ \G(\vc,u,p)= \G_1(\vc,u,p_\vt)+\G_2(z,u,p)\; ,$$
since, for $\G_1$, \TCs reduces to \TC and $\G_2$ readily satisfies \TCs.

Another possibility is to assume that $v(\vc)=v(\vt,\vn)$ is Lipschitz continuous in $\vn$ in $B_\infty (\vcb,r)$, uniformly in $\vt$, i.e. there exists a constant $K>0$ such that, for any $\vc_1=(\vt,\vn_1), \vc_2=(\vt,\vn_2) \in B_\infty (\vcb,r)$
\begin{equation}\label{Lip-z-v}
|v(\vc_1)-v(\vc_2)|\leq K|\vn_1-\vn_2|\; .
\end{equation}

The result for the supersolutions is the
\begin{proposition}\label{reg-by-ic}\emph{--- Regularizations of supersolutions.}\smsp
    Let $v$ be a bounded supersolution of \eqref{gen-eqn-F} and
    assume that\\[2mm]
    $(a)$ either \TCs and \Mong hold\\[2mm]
    $(b)$ or \TC, \Mong and \eqref{Lip-z-v} hold.\\[2mm]
Then there exists a sequence $(v^\e)_\e$ defined in $B_\infty (\vcb,r-a(\e))$ where $a(\e) \to 0$ as $\e\to 0$ such that
    \begin{enumerate}
        \item[$(i)$] each $v^\e$ is a supersolution of \eqref{gen-eqn-F} in $B_\infty (\vcb,r-a(\e))$,
        \item[$(ii)$] each $v^\e$ is semi-concave in the $\vt$-variable,
        \item[$(iii)$] $\limiinf v^\e = v$ as $\e\to 0$.
    \end{enumerate}
\end{proposition}

Two remarks on this proposition: first, the proof is readily the same as for subsolutions, the only
difference is that we do not need to control the $ \vn$-derivative in case $(a)$ because of the
form of \TCs while it is clearly bounded in case $(b)$ because of \eqref{Lip-z-v}. 
The second remark is that, a priori, the $v^\e$ are not continuous in $\vn$ in case $(a)$. But of
course, they are Lipschitz continuous in $\vt$ and $\vn$ in case $(b)$.

\section{Sub and superdifferentials, inequalities at the boundary}
\label{sect:eqn-on-b}\index{Equations at the boundary}

We conclude this chapter with several results concerning the properties of viscosity sub and
supersolutions of an HJ Equation at the boundary of the domain where the equation is set. 
Those results will be mainly applied in Part~\ref{part:NA} but we formulate both in a quite general way
here, considering a general HJ Equation of the form
\begin{equation}\label{genHJ-Om}
u_t+H (x,t,u,Du) = 0 \quad\hbox{in  }Q\; , 
\end{equation}
where $Q:=\Omega \times (0,\Tf)$, $H$ is a continuous function and $\Omega$ is a $C^1$-domain of $\R^N$.
We also set $\partial_\ell Q:= \domeg \times (0,\Tf)$ and $\overline{Q}^\ell = \Omegb\times (0,\Tf)$.

The first result is used below in the proof of Proposition~\ref{prop:equivFL-LS}: in terms of control,
it means that viscosity subsolution inequalities hold up the boundary for all dynamics which are
pointing inward the domain. Here, $d(z)=\dist(z,\partial\Omega)$ denotes the distance to the
boundary  which is $C^1$ in a neighborhood of $\domeg$.

\begin{proposition}\label{sub-up-to-b}\emph{--- Viscosity inequalities at the boundary.}\smsp
    Assume that $u$ is an \usc, locally bounded function on $\overline{Q}^\ell$ which is a subsolution of \eqref{genHJ-Om}. 
    If there exists $(x,t)\in \partial_\ell Q$ and $r>0$ such that
    \begin{enumerate}
        \item[$(i)$]  The \usc function $u$ is $Q$-regular on $\partial_\ell Q
            \cap [B(x,r)\times (t-r,t+r)]$.
    \item[$(ii)$] The distance function $d$  to $\domeg$ is smooth in $\Omegb \cap B(x,r)$,
    \item[$(iii)$] There exists a function $L: \overline{Q}^\ell \cap [B(x,r)\times (t-r,t+r)]
        \times \R \times \R^N \to \R$ such that $L\leq H$ on $\overline{Q}^\ell  \cap
            [B(x,r)\times (t-r,t+r)] \times \R \times \R^N$ and 
            $$ \lambda \mapsto L(y,s,u,p+\lambda Dd(y))\; ,$$
        is a decreasing function for any
        $(y,s,u,p)\in \overline{Q}^\ell \cap [B(x,r)\times (t-r,t+r)] \times \R \times \R^N$.
    \end{enumerate}
    Then $u$ is a subsolution of 
    $$ u_t+L (x,t,u,Du) = 0 \quad\hbox{on  }\partial_\ell Q
            \cap [B(x,r)\times (t-r,t+r)]\; . $$
    Moreover, if we can take $L=H$ the same result is valid for supersolutions.
\end{proposition}

We point out that this result holds for ``regular subsolutions'', i.e. which satisfy $(i)$, a
regularity which is a consequence of Proposition~\ref{reg-sub} if we have suitable normal
controllability and tangential continuity type assumptions.

\begin{proof}
    We consider a test-function $\psi$ which is $C^1$ on $\overline{Q}^\ell$ and we assume that $(y,s)
    \in \partial_\ell Q  \cap [B(x,r)\times (t-r,t+r)]$ is a strict local maximum point of $u-\psi$ (again we refer to
    Lemma~\ref{lem:strict.point} to see why we can always assume the maximum point to be strict). To prove
    the $L$-inequality, we consider the function
    $$ (z,\tau)\mapsto u(z,\tau) - \psi(z,\tau)-\frac{\alpha}{d(z)}\; ,$$
    where $\alpha >0$ is a parameter devoted to tend to $0$. 

    We apply Lemma~\ref{lem:cv-pen} with
    $$\begin{aligned}
        & w(z,\tau) := u(z,\tau)-\psi(z,\tau)\;,\ \chi_\alpha(z,\tau)=\frac{\alpha}{d(z)}\;,\\
        & K=F=\overline{Q}^\ell  \cap (\overline{B(x,r)}\times[t-r,t+r]) \;.
    \end{aligned}$$
    Assumption $(i)$ of Lemma~\ref{lem:cv-pen} is clearly satisfied and since 
    $\liminf_*\chi_\alpha=0$ in $K$ (even on $\domeg$), Assumption $(ii)$ also holds.
    We now turn to condition $(iii)$ which requires some explanations.

    By the $Q$-regularity of $u$ on $\partial_\ell Q  \cap [B(x,r)\times (t-r,t+r)]$, there exists a 
    sequence $(y_k,t_k)$ converging to $(y,t)$ such that $u(y_k,t_k)\to u(y,t)$ and $y_k \in \Omega$.
    We may assume without loss of generality that $d(y_k)\geq k^{-1/2}$. 

    Then, considering the sequence $(y^\alpha,t^\alpha):= (y_{[\alpha^{-1}]},t_{[\alpha^{-1}]})$
    where $[\alpha^{-1}]$ is the integer part of $\alpha^{-1}$, we have $(y^\alpha,t^\alpha)\to
    (y,t)$. Moreover, since $d(y^\alpha)\geq[\alpha^{-1}]^{1/2}$, we deduce also that
    $\chi_\alpha(y^\alpha,t^\alpha)\to 0$ and $w(y^\alpha,t^\alpha)\to w(y,t)$.  In other
    words, this sequence corresponds to the sequence $(z_0^\e)_\e$ required in Assumption $(iii)$ of
    Lemma~\ref{lem:cv-pen}.

    Now, for $\alpha$ small enough, this function has a local maximum
    at $(\bar z,\bar \tau)\in K$, depending on $\alpha$ but we drop this
    dependence for the sake of simplicity of notations. The strict maximum property at
    $(y,s)$ implies its uniqueness, hence Lemma~\ref{lem:cv-pen} ensures that up to
    extraction, as $\alpha \to 0$ we get
    $$ (\bar z,\bar \tau)\to (y,s)\; ,\; u(\bar z,\bar \tau)\to u(y,s)\;.$$

    Writing the viscosity subsolution inequality for $u$, we have 
    $$\psi_t(\bar z,\bar \tau) +H(\bar z,\bar \tau,u(\bar z,\bar \tau),D\psi(\bar z,\bar \tau)-
    \frac{\alpha}{[d(\bar z)]^2}Dd(\bar z) ) \leq  0\; ,$$
    which implies that the same inequality holds for $L$ since $L\leq H$.
    Finally we use the monotonicity property of $L$ in the $Dd(y)$-direction which yields
    $$\psi_t(\bar z,\bar \tau) +L (\bar z,\bar \tau,u(\bar z,\bar \tau),D\psi(\bar z,\bar \tau) )
    \leq  0\; .$$
    The conclusion follows by letting $\alpha$ tends to $0$, using the continuity of $L$.

    For the supersolution property, we argue in an analogous way, looking at a minimum point and
    introducing a ``$\displaystyle +\frac{\alpha}{d(z)}$'' term instead of the ``$\displaystyle
    -\frac{\alpha}{d(z)}$''-one. 
\end{proof}

Then we turn to the classical notions of sub and superdifferentials: we describe their properties on the boundary
$\partial_\ell Q$ since those on $Q$ are well-known and, as we already mentioned it above, some of these properties play a 
crucial role in Part~\ref{part:NA}. Here we add the term ``relatively to $\overline{Q}^\ell$'' since, in the sequel,
we are going to consider at least two domains
with a common boundary. Therefore, on $\partial_\ell Q$ we can consider both sub and
super-differentials relatively either to $\overline{Q}^\ell$ or to its complementary.

We first give the general definition for any point in $\overline{Q}^\ell$.

\begin{definition}\label{def:sub-superd}\emph{--- Sub/superdifferentials relatively to
    $\overline{Q}^\ell$.}\nobreak
    \begin{enumerate}
        \item[$(i)$] The superdifferential relatively to $\overline{Q}^\ell$ of an \usc function
            $u:\overline{Q}^\ell  \to \R$
        at a point $(\xb,\tb) \in \overline{Q}^\ell$ is the, possibly empty,
            closed convex set $D_{\overline{Q}^\ell}^+ u (\xb,\tb)\subset \R^{N+1}$, defined by: $(p_x,p_t)\in
            D_{\overline{Q}^\ell}^+ u (\xb,\tb)$ if and only if, for any $(x,t)\in \overline{Q}^\ell$,
                        $$ u (x,t)\leq u (\xb,\tb) +p_x\cdot(x-\xb) + p_t(t-\tb) + o(|t-\tb|+|x-\xb|)\; ,$$
        \item[$(ii)$] The subdifferential relatively to $\overline{Q}^\ell$ of an \lsc function
            $v:\overline{Q}^\ell \to \R$
        at a point $(\xb,\tb) \in \overline{Q}^\ell$ is the, possibly empty, closed
            convex set $D_{\overline{Q}^\ell}^-v (\xb,\tb)\subset \R^{N+1}$, defined by: $(p_x,p_t)\in
            D_{\overline{Q}^\ell}^-v (\xb,\tb)$ if and only if, for any $(x,t) \in \overline{Q}^\ell$, 
            $$ v (x,t)\geq v (\xb,\tb) +p_x\cdot(x-\xb) + p_t(t-\tb) + o(|t-\tb|+|x-\xb|)\; .$$
    \end{enumerate}
\end{definition}

Of course, the terminology ``relatively to $\overline{Q}^\ell$'' only makes sense for points $(\xb,\tb) \in
\partial_\ell Q$ and if $u$ (or $v$) is defined not only on $\overline{Q}^\ell \times (0,\Tf)$ but on a
larger domain, typically $\R^N \times (0,\Tf)$. Moreover, for points in $Q$,
Definition~\ref{def:sub-superd} is the classical definition. 

The first lemma is classical and we leave its proof to the reader.
\begin{lemma}\label{subdiff-hp-Omega}\emph{--- Sub/superdifferentials on $\overline{Q}^\ell$ and
    test-functions.}
    \begin{enumerate}
        \item[$(i)$] Let $u:\overline{Q}^\ell \to \R$ be an \usc function and $(\xb,\tb)\in \overline{Q}^\ell$. An element $(p_x,p_t)$ is in
            $D_{\overline{Q}^\ell}^+ u (\xb,\tb)$ if and only if there exists a $C^1$-function $\varphi$ such
            that $(\xb,\tb)$ is a strict local maximum point of $u-\varphi$ on $\overline{Q}^\ell$
            and $D_x\varphi (\xb,\tb) =p_x$, $\varphi_t (\xb,\tb) =p_t$.
        \item[$(ii)$] Let $v:\overline{Q}^\ell \to \R$ be an \lsc function and $(\xb,\tb)\in \overline{Q}^\ell$. An element $(p_x,p_t)$ is in
            $D_{\overline{Q}^\ell}^-v (\xb,\tb)$ if and only if there exists a $C^1$-function $\varphi$ such
            that $(\xb,\tb)$ is a strict local minimum point of $u-\varphi$ on $\overline{Q}^\ell$
            and $D_x\varphi (\xb,\tb) =p_x$, $\varphi_t (\xb,\tb) =p_t$.
    \end{enumerate}
\end{lemma}

We have formulated Lemma~\ref{subdiff-hp-Omega} with ``strict'' local maximum or minimum point but,
obviously, this is a fortiori true with just local maximum or minimum. 

Now we turn to the structure of the sub and superdifferentials on the boundary and the connections 
with Equation~\eqref{genHJ-Om}. With the notations of Proposition~\ref{sub-up-to-b}, we have

\begin{proposition}\label{sub-ineq-on-b}\emph{--- Structure of the sub and superdifferentials 
    on $\partial_\ell Q$ and inequalities up to the boundary.}
    \begin{enumerate}
        \item[$(i)$] Assume that $u:\overline{Q}^\ell \to \R$ is an \usc, locally bounded subsolution of \eqref{genHJ-Om} which is 
        $Q$-regular at the point $(x,t)\in \partial_\ell Q$.
        If $(p_x,p_t) \in D_{\overline{Q}^\ell}^+ u (x,t)$, then the set $I= \{\lambda \in \R: (p_x+\lambda
            Dd(x),p_t) \in D_{\overline{Q}^\ell}^+ u (x,t)\}$ is an interval, either $I=\R$ or
            $I=[\underline\lambda,+\infty)$ for some $\underline\lambda \leq 0$ and in this latter case,
            $$ p_t +H (x,t,u(x,t),p_x+\underline\lambda Dd(x))\leq 0\; .$$
        \item[$(ii)$] Assume that $v:\overline{Q}^\ell \to \R$ is a \lsc, locally bounded supersolution of \eqref{genHJ-Om}
            which is $Q$-regular at the point $(x,t)\in \partial_\ell Q$.
            If $(p_x,p_t) \in D_{\overline{Q}^\ell}^- v(x,t)$, then the set $J= \{\lambda \in \R: (p_x+\lambda
            Dd(x),p_t) \in D_{\overline{Q}^\ell}^-v (x,t)\}$ is an interval, either $J=\R$ or
            $J=(-\infty,\overline\lambda]$ for some $\overline\lambda\geq 0$ and in this latter case, 
            $$p_t + H (x,t,v(x,t),p_x+\overline\lambda Dd(x))\geq 0\; .$$
    \end{enumerate}
\end{proposition}

We recall that, for $x\in \domeg$, $Dd(x)$ is the unit normal vector to $\domeg$ at $x$ pointing
inward $\Omega$. Therefore Proposition~\ref{sub-ineq-on-b} gives informations on the structure of
the sub and superdifferentials on the boundary in the normal direction.

\begin{proof}
    We provide a complete proof only in the subsolutions case, the case of supersolutions follows
    from similar arguments.

    \noindent \textbf{(a)} 
    Since $D_{\overline{Q}^\ell}^+ u (x,t)$ is a non-empty closed convex subset of $\R^{N+1}$, it is clear that
    $I$ is also a non-empty closed convex subset of $\R$, hence an interval. Moreover, we claim that
    since $(p_x,p_t) \in D_{\overline{Q}^\ell}^+ u (x,t)$, then also $(p_x+\lambda  Dd(x),p_t) \in
    D_{\overline{Q}^\ell}^+
    u (x,t)$ for any $\lambda \geq 0$. Indeed, if $y\in \Omegb$, by the regularity of $d(\cdot)$, 
    $$ 0\leq d(y)=d(x)+Dd(x)\cdot (y-x)+ o(|y-x|)=Dd(x)\cdot(y-x)+o(|y-x|)\;.$$
    So, for any $\lambda\geq0$, $\lambda Dd(x)\cdot (y-x)\geq o(|y-x|)$ and the claim follows. 
    Hence $I$ is either $\R$ or of the form $[\underline\lambda,+\infty)$ for some
    $\underline\lambda\in\R$, and necessarily $\underline{\lambda}\leq 0$ because $\lambda=0\in I$.

    It remains to prove the viscosity inequality when $\underline{\lambda}>-\infty$. 
    
    \smallskip

    \noindent \textbf{(b)} Since $(p_x+\underline \lambda Dd(x),p_t) \in
    D_{\overline{Q}^\ell}^+ u (x,t)$, by Lemma~\ref{subdiff-hp-Omega}, there exists a $C^1$-function $\varphi$ such
    that $(x,t)$ is a strict local maximum point of $u-\varphi$ on $\overline{Q}^\ell$ and
    $D_x\varphi (x,t) =p_x+\underline \lambda Dd(x)$, $\varphi_t (x,t) =p_t$. Then, for $0<\e\ll1$,
    we consider the function
    $$ \psi_\e (y,s) = u(y,s)-\varphi(y,s)+\e d(y)\; .$$
    Since $(x,t)$ is a strict local maximum point of $u-\varphi$ on $\overline{Q}^\ell$, for $\e$
    small enough, there exists a maximum point $(\ye,\se)$ of $\psi_\e$ near $(x,t)$ and we have
    $(\ye,\se)\to (x,t)$, $u(\ye,\se) \to u(x,t)$ as $\e \to 0$.

    We claim that $(\ye,\se) \in Q$, at least for $\e$ small enough. 
    Indeed, if $(\ye,\se) \in \partial_\ell Q $, then necessarily $ (\ye,\se) =(x,t)$; otherwise,
    by the strict maximum point property, we would have $$\psi_\e (\ye,\se)  = (u-\varphi)(\ye,\se)
    < (u-\varphi) (x,t)=\psi_\e (x,t)$$ which would contradict the maximality of $(\ye,\se)$ for
    $\psi_\e$. But $ (\ye,\se) =(x,t)$ is not possible since it would imply that $(p_x+(\underline
    \lambda-\e) Dd(x),p_t) \in D_{\overline{Q}^\ell}^+ u (x,t)$, a contradiction to the minimality of
    $\underline \lambda$.

    \smallskip

    \noindent \textbf{(c)} Therefore $(\ye,\se) \in Q$ and the viscosity subsolution
    inequality holds, namely $$ \varphi_t (\ye,\se) +H (\ye,\se,u(\ye,\se),D_x\varphi
    +(\underline\lambda-\e) Dd(\ye))\leq 0\; .$$ 
    The result follows by letting $\e\to 0$, using the continuity of $H$ and
    the fact that $\varphi$ is~$C^1$.
\end{proof}

\begin{remark}\emph{--- Sub and superdifferentials on $\partial_\ell Q$ and regularity.}\smsp
In Proposition~\ref{sub-ineq-on-b}, we assume the sub and supersolutions to be $Q$-regular at the
    point $(x,t)\in \partial_\ell Q$: this
is, of course, to obtain the viscosity inequalities for $\underline\lambda$ and $\overline\lambda$. We point out anyway that\\
-- even if these regularity properties hold, $\underline\lambda$ and $\overline\lambda$ can be infinite. Take $Q=(0,+\infty)\times (0,\Tf)$
and consider the functions $u(x,t)=-x^{1/2}$ or $v(x,t)=x^{1/2}$.\\
-- If $u$ is NOT $Q$-regular at the point $(x,t)\in \partial_\ell Q$ and if $D_{\overline{Q}^\ell}^+ u (x,t)$ is non-empty then $I=\R$ and, in the 
same way, if $v$ is NOT $Q$-regular at the point $(x,t)\in \partial_\ell Q$ and if
    $D_{\overline{Q}^\ell}^- v (x,t)$ is non-empty then $J=\R$.
\end{remark}

We conclude this section by a ``two-domain'' result. More precisely we consider a domain $\tilde
\Omega \subset \R^N$ which can be written as $$\tilde \Omega= \Omega_1\cup \Omega_2 \cup \H\; ,$$
where $\Omega_1, \Omega_2$ are two disjoints domains of $\R^N$ and $\H=\partial \Omega_1 \cap \partial
\Omega_2$ is a smooth $(N-1)$-manifold. We use the notations $Q_i=\Omega_i\times (0,\Tf)$,
$\Omegb_i^\ell =\Omegb_i\times (0,\Tf)$ and we notice that $\H \times (0,\Tf) \subset
\partial_\ell \Omega_i$ for $i=1,2$. Finally we
denote by $d(\cdot)$ the distance function to $\H$ and by
$n(x)$ the unit normal vector to $\H$ pointing inward to $\Omega_1$.

Given $\Lambda=(\lambda_1,\lambda_2)\in\R$, we define the continuous function 
$\chi^\Lambda:\tilde \Omega \to \R$ by 
            $$\chi^\Lambda(x)=\begin{cases} \lambda_1d(x) & \hbox{if $x \in \overline \Omega_1$,}\\
            \lambda_2 d(x) & \hbox{if $x\in \overline \Omega_2$,} \end{cases}$$ 

\begin{lemma}\label{diff-twod}\emph{--- Sub and superdifferentials on $\H \times (0,\Tf)$ and
    test-functions.}
    \begin{enumerate}
        \item[$(i)$] Let $u:\tilde \Omega \times (0,\Tf)\to \R$ be an \usc function and $(\xb,\tb)$ a point
            of $\H\times(0,\Tf)$. We assume that there exists $(p_x,p_t)\in \R^{N+1}$ and
            $\Lambda=(\lambda_1,\lambda_2)\in \R^2$ such that 
            $$\begin{cases}
                (p_x+\lambda_1n(x),p_t)\in D_{\Omegb_1^l}^+ u (\xb,\tb)\;,\\ 
                (p_x-\lambda_2 n(x),p_t)\in D_{\Omegb_2^l}^+ u (\xb,\tb)\;.
            \end{cases}$$
            Then there exists a $C^1$-function $\varphi$ such that $(\xb,\tb)$ is a strict local
            maximum point of $u-\chi^\Lambda-\varphi$ on $\tilde \Omega \times (0,\Tf)$ and $D_x\varphi
            (\xb,\tb) =p_x$, $\varphi_t (\xb,\tb) =p_t$.

        \item[$(ii)$] Let $v:\tilde \Omega \times (0,\Tf)\to \R$
            be a \lsc function and $(\xb,\tb)$ a point of $\H\times(0,\Tf)$. We assume that there
            exists $(p_x,p_t)\in \R^{N+1}$ and $\Lambda=(\lambda_1,\lambda_2)\in \R^2$ such that 
            $$\begin{cases}
             (p_x+\lambda_1n(x),p_t)\in D_{\Omegb_1^l}^-v (\xb,\tb)\,\\
             (p_x-\lambda_2 n(x),p_t)\in D_{\Omegb_2^l}^-v(\xb,\tb)\;.
            \end{cases}$$
            Then there exists a $C^1$-function
            $\varphi$ such that $(\xb,\tb)$ is a strict local maximum point of $u-\chi^\Lambda-\varphi$ on
            $\tilde \Omega \times (0,\Tf)$ and $D_x\varphi (\xb,\tb) =p_x$, $\varphi_t (\xb,\tb)
            =p_t$.
    \end{enumerate}
\end{lemma}

We refer the reader to Part~\ref{part:NA} where we introduce test-functions which are piecewise
$C^1$ like $\chi^\Lambda+\varphi$ above. Lemma~\ref{diff-twod} will be useful in this context.

\begin{proof}
    The proof is short and we provide it only in the subdifferential case, the proof for the
    superdifferential being analogous. We just notice that $(p_x,p_t)$ is in the super-differential
    of the \usc function $u-\chi^\Lambda$ at $(\xb,\tb)$. The existence of $\varphi$ is therefore a
    consequence of the classical results on subdifferentials.  
\end{proof}


\chapter{Control Tools}
\fancyhead[CO]{HJ-Equations with Discontinuities: Control Tools}
\label{chap:control.tools}

\abstract{Classical results for control problems with discontinuities are presented. This is the
occasion to describe the classical framework that is used in particular in the stratified
setting. The main results concern the connections between viscosity sub and supersolutions, and the
associated sub and super-dynamic programming principles.}

Of course, the key ingredients used in this chapter are not new, we just try to revisit them in a
more modern way: we refer the reader to the founding article of Filippov \cite{MR0149985} and to
Aubin and Cellina \cite{MR755330}, Aubin and Frankowska \cite{AF}, Clarke \cite{MR1058436}, Clarke,
Ledyaev, Stern and Wolenski \cite{MR1488695} for the classical approach of deterministic control
problems by non-smooth analysis methods.

\section{Introduction: how to define deterministic control problems with discontinuities?  The two half-spaces problem}
\label{sect:disc.pb}

\index{Control problem!two-half-spaces}
As in the basic example of a two half-space discontinuity introduced in Section~\ref{sect:stab}, we
consider a partition of $\R^N$ into 
$$\H=\{x_N=0\}\;,\ \Omega_1=\{x_N>0\}\;,\ \Omega_2=\{x_N<0\} \; ,$$
and, following Chapter~\ref{chap:BasicFram}, we assume that we are given three different control problems
in each of these subsets associated to dynamics, discount factors and costs respectively denoted by
$(b_0,c_0,l_0), (b_1,c_1,l_1),(b_2,c_2,l_2)$; hence, the Hamiltonians take the form
$$
        H_i(x,t,r,p_x):= \sup_{\alpha_i \in A_i}\,\left\{-b_i(x,t, \alpha)\cdot p_x + 
        c_i(x,t, \alpha)r-l_i(x,t,\alpha)\right\}\; ,
    $$ 
for $i=0,1,2$, where the $A_i$ are the spaces of controls. For the sake of simplicity, we can assume
that the $(b_i,c_i,l_i)$ are all defined on $\R^N\times[0,\Tf]\times A_i$ for $i=0,1,2 $ and even
that they all satisfy \HCP. As a consequence, the $H_i$ are well-defined and continuous in
$\R^N\times[0,\Tf]\times \R \times \R^N$.

For such problems, the first question consists in defining properly the global dynamic $b$ since,
when the trajectory reaches $\H$, a discontinuity in the dynamic occurs: the controller may have
access to dynamics $b_1$ and $b_2$, but also to the specific dynamics $b_0$. Of course,
the similar question of defining globally the discount factor and cost holds.

The natural tool consists in using the theory of {\em differential inclusions} that we first
introduce on the simple example of Chapter~\ref{chap:BasicFram}. The idea consists in looking at the
set valued map 
$$\BCL(x,t):=\{ (b(x,t, \alpha),c(x,t, \alpha),l(x,t, \alpha))\ : \ \alpha \in A\}\;,$$
and to solve the differential inclusion
$$ (\dot X (s),\dot D (s),\dot L (s)) \in \BCL (X(s),t-s)\; , \; (X,D,L)(0)=(x,0,0)\; ,$$
which only required that the set valued map $\BCL$ is upper-semicontinuous, with values in compact,
convex sets (which is almost satisfied here, at least, adding the assumptions that the $\BCL(x,t)$
are convex or solving with their convex hull). Then 
$$
 \tilde U(x,t) = \inf_{(X,D,L)} \bigg(\int_0^{t} \dot L(s) \exp(D(s))) \ds + u_0(X(t))\exp(D(t) )
 \bigg)\; ,
$$
The advantage of this approach is to allow to define the dynamic, discount and cost without any
regularity in $b,c,l$. 

The next step is the half-space discontinuity for which we are going to define $\BCL$ in the same
way for $x \in \Omega_1$ and $x\in \Omega_2$ by just setting, for $t\in [0,\Tf]$
$$
\begin{cases}
    (b(x,t,\alpha) ,c(x,t,\alpha),l(x,t,\alpha)) = 
    (b_1(x,t,\alpha_1),c_1(x,t,\alpha_1),l_1(x,t,\alpha_1))& \hbox{if $x\in \Omega_1$}\\
    (b(x,t,\alpha) ,c(x,t,\alpha),l(x,t,\alpha)) = 
    (b_2(x,t,\alpha_2),c_2(x,t,\alpha_2),l_2(x,t,\alpha_2))& \hbox{if $x\in \Omega_2$}
\end{cases}
$$
where $\alpha \in A=A_0 \times A_1 \times A_2$, the ``extended control space''.

For $x \in \H$ and $t\in [0,\Tf]$, we just follow the theory of differential inclusions: by the
upper semi-continuity of $\BCL$, we necessarily have in $\BCL(x,t)$ all the
$(b_i(x,t,\alpha_i),c_i(x,t,\alpha_i),l_i(x,t,\alpha_i))$ for $i=0,1,2$ but we have also to take the
convex hull of all these elements, namely all the convex combinations of them.  In particular, for
the dynamic, we have (a priori) all the $b=\mu_0 b_0+\mu_1 b_1+\mu_2 b_2$ such that
$\mu_0+\mu_1+\mu_2=1$, $\mu_i \geq 0$ but we will show that such $b$ play a role only if the
trajectory stays on $\H$ and therefore if we have $b\cdot e_N=0$. A more precise statement will be
given in Section~\ref{sect:codimIa}.

\section{A general framework for deterministic control problems}
\label{Gen-DCP}

\index{Control problem!general}
Based on the ideas that we sketched in last section, we consider a general approach of finite
horizon control problems with differential inclusions. We use an \emph{extended trajectory}
$(X,T,D,L)$ in which we also embed the running time variable $T$, pointing out that, in the basic
example we introduced in the previous section, we just have $T(s)=t-s$.
 
This framework may seem complicated but we made this choice because it allows us to consider all the
applications we have in mind: on one hand, time and space will play analogous roles when we face
time-dependent discontinuities, or for treating some unbounded control type features; on the
other hand, discount factors will be necessarily involved when dealing with boundary conditions---\,see Part~\ref{S-BC}. 

In this section, we present general and classical results which do not require any particular
assumption concerning neither the structure of the discontinuities, nor on the control sets.

In the following, we denote by $\mathcal{P}(E)$ the set of all subsets of $E$.

\subsection{Dynamics, discounts and costs} 
\label{sec:HBCLa}

The first hypothesis we make is

\label{page:HBCLa}\label{not:BCL}
\begin{assumption}{\HBCLa}{Fundamental assumptions on $\BCL$.}
    The set-valued map
    $\BCL:\R^N\times[0,\Tf]\to\mathcal{P}(\R^{N+3})$ satisfies
    \begin{enumerate}
    \item[$(i)$] the map $(x,t)\mapsto\BCL(x,t)$ has compact, convex images
      and is upper semi-continuous;
    \item[$(ii)$] there exists $M>0$, such that, for any $x\in\R^N$ and $t\in [0,\Tf]$,
      $$\BCL(x,t)\subset \big\{(b,c,l)\in\R^{N+1}\times\R\times\R:|b|\leq M; |c|\leq M; | l | \leq
          M\big\}\;.$$
  \end{enumerate}
\end{assumption}
Here, $|\cdot|$ stands for the usual euclidian norm in any euclidean space $\R^p$ (which reduces to
the absolute value in $\R$, for the $c$ and $l$ variables). If $(b,c,l) \in \BCL(x,t)$, $b$
corresponds to the dynamic (both in space and time), $c$ to the discount factor and $l$ to the
running cost. Assumption \HBCLa-$(ii)$ means that dynamics, discount factors and running costs are
uniformly bounded.  In the following, we sometimes have to consider separately dynamics, discount
factors and running costs. To do so, we set
$$
\B(x,t)=\big\{b \in\R^{N+1};\ \hbox{there exists $c,l\in\R$ such that  } (b,c,l)\in\BCL(x,t) \big\}\; ,
$$
and analogously for $\C (x,t),\L(x,t)\subset\R$. Finally, we decompose any $b\in \B (x,t)$ as
$(b^x,b^t)$, where $b^x$ and $b^t$ are respectively the space and time dynamics.

We recall the definition of upper semi-continuity we use here: 
a set-valued map $x\mapsto F(x)$ is upper-semi continuous at $x_0$ if for any
open set $\mathcal{O}\supset F(x_0)$, there exists an open set $\omega$
containing $x_0$ such that $F(\omega)\subset\mathcal{O}$. Expressed in other terms, 
$F(x)\supset \limsup\limits_{y\to x} F(y)$.

\subsection{The control problem} 

\label{not:XTDL}
We look for trajectories $(X,T,D,L)(\cdot)$ of the differential inclusion
\begin{equation}\label{eq:diff.inc}
\begin{cases}
\dfrac{\d}{\dt}(X,T,D,L)(s)\in\BCL\big(X(s),T(s)\big)&\ \text{ for a.e. }s\in[0,+\infty)\;,\\[2mm]
(X,T,D,L)(0)=(x,t,0,0)\;.
\end{cases}
\end{equation}
The key existence result is the

\begin{theorem}\label{thm:existence.traj}
	Assume that \HBCLa holds. Then \\[2mm]
	$(i)$ for any $(x,t)\in\R^N\times[0,\Tf)$ there exists a Lipschitz function 
    $(X,T,D,L):[0,\Tf]\to\R^N\times\R^3$
	which is a solution of the differential inclusion \eqref{eq:diff.inc}.\\[2mm]
	$(ii)$ for each solution $(X,T,D,L)$ of \eqref{eq:diff.inc} there exist 
    measurable functions $(b,c,l)(\cdot)$ such that for a.e. $s\in(t,\Tf)$,
	$$(\dot X,\dot T,\dot D, \dot L)(s)=(b,c,l)(s)\in\BCL(X(s),T(s))\;.$$
\end{theorem}

Throughout this chapter, we mostly write 
$$
    \begin{aligned}
    (\dot X(s), \dot T(s)) &= b\big(X(s),T(s)\big)\\
    \dot D(s) &= c\big(X(s),T(s)\big)\\
    \dot L(s) &= l\big(X(s),T(s)\big)
    \end{aligned}
$$
in order to remember that $b$, $c$ and $l$ correspond to a specific choice in $\BCL(X(s),T(s))$.
Later on, we will also introduce a control $\alpha(\cdot)$ to represent the $(b,c,l)$ as
$$(b,c,l)(X(s),T(s),\alpha(s))\;.$$

In order to simplify the notations, we just use the notation $X,T,D,L$ when there is no ambiguity
but we may also use the notations $X^{x,t},T^{x,t}, D^{x,t},L^{x,t}$ when the dependence in
$x,t$ plays an important role.

Let us introduce a point of vocabulary here: by a state-constrained control problem in a set
$\mathcal{W}$, we mean that the controller can only use trajectories which remain in $\mathcal{W}$:
$(X,T,D,L)(s)\in\mathcal{W}$ for any $s\in[0,+\infty)$. In general, such constraint only concerns
the state variable $X$, which is required to satisfy $X(s) \in \Omegb$ for some domain $\Omega$: 
we study these state-constrained problems in Part~\ref{S-BC}.  

However, throughout this book we have chosen a framework with a dynamic on $T$ in order to describe
finite horizon control problems in $\R^N \times [0,\Tf]$ (or $\Omegb \times [0,\Tf]$). Hence,
the $T$-variable is also constrained to satisfy $T(s)\in[0,\Tf]$. This property is at the origin of
some of the hypotheses below. In this setting, the usual terminal cost is changed into a running
cost, which also requires some assumptions in order to have a bounded value function.

Before describing the value function, we are going to make the following structure assumptions 
on the $\BCL$-set valued map

\label{page:HBCLb}
\begin{assumption}{\HBCLb}{Structure assumptions on the $\BCL$.}
    There exists $\uc, K >0$ such that
\begin{enumerate}
\item[$(i)$] For all $x\in \R^N$, $t\in [0,\Tf]$ and $b=(b^x,b^t) \in \B (x,t)$, 
    $-1 \leq b^t \leq 0$. Moreover, there exists $b=(b^x,b^t) \in \B (x,t)$ such that $b^t =-1$.
\item[$(ii)$] For all $x\in \R^N$, $t\in [0,\Tf]$,  if $((b^x,b^t),c,l) \in \BCL (x,t)$, 
    then $-Kb^t + c \geq 0$. 
\item[$(iii)$] For any $x \in \R^N$, there exists an element in $\BCL(x,0)$ 
    of the form $((0,0),c,l)$ with $c \geq \uc$.
\item[$(iv)$] For all $x\in \R^N$, $t\in [0,\Tf]$, if $(b,c,l) \in \BCL (x,t)$ 
    then $\max(-b^t,c,l)\geq \uc$.
\end{enumerate}
\end{assumption}

By introducing  this general framework, our aim is to gather different type of control problems and
treat them within the same setting. In classical finite horizon problems $b^t=-1$, which indicates
a time direction associated to the $u_t$-term, and in this case $T(s)=t-s$.  Here we choose the more
general assumption $-1\leq b^t \leq 0$ in order to respect this monotonicity in time, but allowing
also $b^t=0$ which can corresponds
\begin{enumerate}
\item either to a control problem with a stopping time; in particular, we point out that a classical final cost
is treated as associated to a stopping time control problem.
\item Or an unbounded control problem. The reader may be surprised by this claim since the $b$'s are bounded
but this framework typically contains cases where the cost is proportional to the dynamic, allowing jumps (See, for example,
Section~\ref{sec:eiuc} and the
beginning of Chapter~\ref{chap:jumps}).
\end{enumerate}
Of course, a combination of the two is possible. We point out anyway that unbounded control problems with a cost having a superlinear growth \wrt
the dynamic (typically, a quadratic cost) does not enter a priori in our framework.

Assumption $(iii)$ and a part of $(iv)$ concern the final cost ($u_0$ in the example of the previous
section) which is in general the initial data for the Hamilton-Jacobi-Bellman Equation. As we
pointed out above, the value function we define below is associated to a state-constrained problem in
$\R^N \times [0,\Tf]$, and therefore it is necessary that strategies with $b^t=0$ for any point
$(x,0)\in \R^N\times \{0\}$ exist.

Assumption~$(iii)$ means that we can stop the trajectory at any point $(x,0)$, as for the case of a
classical initial data, the assumption on $c$ being necessary, in general, to keep the integral of
the running cost bounded. However, strategies with $b^t=0$, $b^x\neq 0$ are also allowed provided
that they satisfy $(iv)$ at time $t=0$ in order, again, that the associated cost remains bounded:
indeed, either the trajectory is associated to a positive discount factor $c\geq\underline{c}$ which
ensures the boundedness of the integral of the running cost or it has a positive cost
$l\geq\underline{c}$ in order to avoid the long use of this strategy.

Such situations may also happen for $t>0$, either to model a possible stopping time (obstacle type
problem) or an exit cost (see in Part~\ref{S-BC}, Dirichlet boundary condition), which is why
$(iv)$ is written for all $t\in[0,\Tf]$.

On the other hand, the consequence of $(ii)$ is that the change of unknown function $u\to
\exp(-Kt)u$ allows to reduce to the easier case of a positive discount factor. Such assumption is
necessary in this framework since the formulation below leads to a stationary type equation, because
we treat time as a space variable.

Finally, notice that the fact that $b^t$ can be $0$ (or close to it) includes the unbounded control
case. In particular if $b^t=0$, the trajectory can stay at a constant time $\bar t$ for, say, $s \in
[s_1,s_2]$ while if $b^x\neq0$, the trajectory can be seen as an instantaneous jump from the point
$X(s_1)$ to the point $X(s_2)$ since time does not vary on this interval.

In all the rest of the book, \HBCL means that both \HBCLa and \HBCLb are fulfilled.

In order to introduce the value function, we state below a result showing that the cost we use is
well-defined and bounded from below. We also provide various properties, among which the fact that
we can always reduce ourselves to the case $c\geq 0$ for any $(b,c,l) \in \BCL(x,t)$ and for any
$(x,t)\in \R^N\times[0,\Tf]$.
\begin{lemma}\label{c-posit}
    Assume that \HBCL holds and let $(X,T,D,L)$ be a solution of \eqref{eq:diff.inc} associated to
    $(b,c,l)(\cdot)$ such that $(X(s),T(s)) \in \R^N\times [0,\Tf]$ for all $s>0$.  Then
    \begin{enumerate}
        \item[$(i)$] The following integral is well-defined and uniformly bounded from below
$$ J(X,T,D,L)=\int_0^{+\infty} l\big(X(s),T(s)\big)\exp(-D(s))\dt\;.$$
\item[$(ii)$] For any trajectory $(X,T,D,L)$ of the differential inclusion such that 
    $$J(X,T,D,L):=\int_0^{+\infty} l\big(X(s),T(s)\big)\exp(-D(s))ds<\infty\;,$$
    then $D(s) \to +\infty$ as $s\to +\infty$. \\
\item[$(iii)$] If $K$ is  the constant given by \HBCLb, we have 
$$ \exp(-Kt)J(X,T,D,L)=J(\tilde X,\tilde T,\tilde D,\tilde L)\; ,$$
where $(\tilde X,\tilde T,\tilde D,\tilde L)$ is the solution of \eqref{eq:diff.inc} associated to
$(b,c-Kb_t,l\exp(-KT(s)))(\cdot)$. In particular $\tilde X=X$, $\tilde T=T$, $\tilde D=D+K(T-t)$ and
    of course we still have $(\tilde X,\tilde T,\tilde D,\tilde L)(0)=(x,t,0,0)$.
    \end{enumerate}
\end{lemma}

The use of this lemma will be clear in the next sections but it is obvious from \HBCLb-$(ii)$ that
the replacement of $c$ by $c-Kb_t\geq0$ allows as we wish to reduce $c\in\R$ to the case when
$c\geq0$.

\begin{proof} We divide it into several steps.\smsp
    \noindent\textbf{(a)} In order to prove the two properties of $(i)$, we use \HBCLb-$(iv)$ and introduce the sets 
    $$ E_1:=\{s: - b^t \geq \uc\,\}\; ,\; E_2:=\{s\notin E_1 : c \geq \uc\,\}\; ,
    \; E_3=[0,+\infty)\setminus (E_1 \cup E_2)\; .$$
    By \HBCLb-$(iv)$, we have $[0,+\infty)=E_1\cup E_2\cup E_3$ and $E_1,E_2,E_3$ are disjoint by
    construction. We now evaluate the integral on each of these three sets.
    
    Concerning the $E_1$-contribution, we notice that, using that $\dot T(s)=b^t$,
    $$|E_1|\,\uc\leq \int_{E_1} -b^t(X(s),T(s))\ds \leq T(0)=t\;.$$ 
    Since $l$ is bounded, $0\leq \exp(-D(s))\leq 1$ and $|E_1|\leq t/\uc$, the function
    $$s \mapsto \1_{E_1}(s)l\big(X(s),T(s)\big)\exp(-D(s))\; ,$$
    is in $L^1(0,+\infty)$ and its contribution---its $L^1$-norm---is uniformly bounded by $M t/\uc$.
    
    On $E_2$, since $\dot 
    D(s)=c(s) \geq \uc$, it follows that 
    $$\begin{aligned}
        \int_{E_2} \vert l\big(X(s),T(s)\big)\vert \exp(-D(s))\ds \leq & M \int_{E_2} \exp(-D(s))ds \\
        \leq & M \int_{E_2} \frac{\dot D(s)}{\uc} \exp(-D(s))\ds\\
         \leq & M \int_{[0,+\infty)} \frac{\dot D(s)}{\uc} \exp(-D(s))\ds \leq \frac{M}{\uc} \;,
    \end{aligned}$$
    Hence we have also that the function
    $$s \mapsto \1_{E_2}(s)l\big(X(s),T(s)\big)\exp(-D(s))\; ,$$
    is in $L^1(0,+\infty)$ and its contribution---its $L^1$-norm---is uniformly bounded by $M/\uc$.
    
    Finally, on $E_3$, we integrate a positive function; therefore the corresponding integral is
    well-defined and bounded from below. This completes the proof of $(i)$.
 
    \noindent\textbf{(b)} In order to prove $(ii)$, we examine carefully the sets $E_1,
    E_2, E_3$ defined above. We recall first that $|E_1|\leq t/\uc<+\infty$, so that
    necessarily, either $E_2$ or $E_3$ has infinite Lebesgue measure.
    Now, on $E_2$, $\dot D(s)=c(s)\geq\uc$ so that
    $$\uc\cdot |E_2\cap[0,S]|\leq\int_{E_2\cap[0,S]}\dot D(s)\ds\leq D(S)\;.$$
    We deduce that if the increasing function $s\mapsto D(s)$ does not tend to $+\infty$ when $s\to
    +\infty$, then $|E_2|\leq \sup_s D(s)/\uc\,<\infty$, so that $|E_3|=+\infty$.

    By the monotonicity of $D$, if $D(s)$ does not tend to $+\infty$ when $s\to +\infty$, there
    exists $\gamma>0$ such that $\exp(-D(s))\geq \gamma$ on $[0,+\infty)$ but on $E_3$, since
    $l(s)\geq \uc$ we see that $$\int_{E_3} l\big(X(s),T(s)\big)\exp(-D(s))\ds \geq \int_{E_3}
    \uc\cdot \gamma \ds = \uc\cdot\gamma\cdot |E_3|=+\infty\; ,$$
    and we reach a contradiction because integral $J(X,T,D,L)$ is bounded.\\
\noindent\textbf{(c)} The proof of $(iii)$ relies on an easy manipulation on the integral and we skip it.
\end{proof}

\subsection{The value function}\label{sec:VF}

\label{not:mT}
Now we introduce the value function which is defined on $\R^N\times [0,\Tf]$ by
\begin{equation}\label{eqn:VF}
U(x,t)=\inf_{\cT(x,t)}\Big\{\int_0^{+\infty} 
     l\big(X(s),T(s)\big)\exp(-D(s))\ds\Big\}\;,
\end{equation}
where $\cT(x,t)$ stands for all the Lipschitz trajectories $(X,T,D,L)$ of the differential inclusion
which start at $(x,t)\in \R^N\times [0,\Tf]$ and such that $(X(s),T(s)) \in \R^N\times [0,\Tf]$ for all
$s>0$.  

As we explained above, Assumption $(iii)-(iv)$ imply formally the existence of trajectories
$(X,T,D,L)$ satisfying the constraint $(X,T) \in \R^N\times[0,\Tf]$ and, by Lemma~\ref{c-posit},
these trajectories are associated to a well-defined cost $J(X,T,D,L)$ which is uniformly bounded
from below. Hence we expect both that $\cT(x,t)\neq \emptyset$ for all $(x,t) \in \R^N \times
[0,\Tf]$ and that $U$ is bounded. A rigorous proof of this claim is contained in the

\begin{lemma}\label{u-bound}
    Assume that \HBCL holds. Then the value function $U$ is bounded on $\R^N \times [0,\Tf]$ and is
    lower semi-continuous in $\R^N\times [0,\Tf]$. Moreover an optimal trajectory exists, \ie
    for any $(x,t)$, there exists a trajectory $(X,T,D,T)\in \cT(x,t)$ such that
$$ U(x,t)=\int_0^{+\infty} 
     l\big(X(s),T(s)\big)\exp(-D(s))\ds\; .
$$ 
\end{lemma}

\begin{proof}
    We first use Lemma~\ref{c-posit} to reduce the proof in the case when $c$ is positive.

    \noindent\textbf{(a)} In order to prove that $U$ is bounded, we first show that $\cT(x,t)\neq \emptyset$.
    Let us solve differential inclusion \eqref{eq:diff.inc}, replacing $\BCL$ by 
    $$\BCL_\flat(x,t):=\BCL(x,t)\cap \{(b,c,l)\in \R^{N+3};\ b^t=-1\}\;.$$ 
    The reader can easily check that this new set-valued map satisfies all the required assumptions
    \HBCLa and \HBCLb. 
    Moreover, for any trajectory associated with $\BCL_\flat$ starting at $(x,t,0,0)$, it is clear
    that $T(t)=0$ since $T(s)=t-s$. Notice that for $s\in[0,t]$, this trajectory may be seen as a
    trajectory associated to the original $\BCL$ since $\BCL_\flat\subset\BCL$.

    Then, for any $s\geq t$ we redefine the trajectory by solving 
    $$(\dot X,\dot T,\dot D, \dot L)(s) = ((0, 0), c, l)$$
    where $((0, 0), c, l)$ is given by Assumption~\HBCLb-$(iii)$---\ie with $c(s)\geq \uc$ for any $s$---for the original $\BCL$,
    at $(x,0)=(X(t),T(t))$. This defines a new trajectory for all $s\in[t,+\infty)$ associated to
    $\BCL$ and obviously, $(X(s),T(s))\in\R^N\times[0,\Tf]$ so that the constructed trajectory 
    $(X,T,D,L)$ belongs to $\cT(x,t)$. Moreover
    $$\int_s^{+\infty} 
     l\big(X(s),T(s)\big)\exp(-D(s))\ds \leq \int_0^{+\infty} 
     M\exp(-\uc(s-t) )\ds \leq \frac{M}{\uc}\; .$$
    Hence, since the contribution on $[0,t]$ is bounded by $M$, $U$ is bounded from above and since
    we know by Lemma~\ref{c-posit} that it is also bounded from below, $U$ is bounded.

    \smallskip

 \noindent\textbf{(b)} In order to show that $U$ is \lsc, we are going to use by anticipation
    Theorem~\ref{DPP}, \ie the Dynamic Programming Principle. Let $(x,t) \in \R^N \times [0,\Tf]$
    and $((\xe,\te))_\e$ a sequence of points of $\R^N \times [0,\Tf]$ which converges to $(x,t)$
    and such that $\lim_\e U(\xe,\te) = \liminf_{(y,s)\to (x,t)} U(y,s)$.
    Our aim is to show that
    $$ \lim_\e U(\xe,\te)\geq U(x,t)\; .$$
 
    By definition of $U$, there exists a trajectory $(X_\e, T_\e,D_\e,L_\e)$ such that
    $$U(\xe,\te)\geq \int_0^{+\infty} l\big(X_\e(s),T_\e(s)\big)\exp(-D_\e(s)) \ds-\e\; .$$
    Using that the $\BCL$-sets are uniformly bounded, we can apply Ascoli-Arzela Theorem together
    with a diagonal extraction procedure to the family of trajectories $(X_\e, T_\e,D_\e,L_\e)$ to
    show that $$ (X_\e, T_\e,D_\e,L_\e)\to (X, T,D,L)\quad\hbox{locally uniformly
    on  }[0,+\infty)\; ,$$
    where $(X,T)$ remains in the domain $\R^N\times [0,\Tf]$. We may also assume that $\dot L_\e =
    l\big(X_\e,T_\e\big)$ weakly converges in the $L^\infty$-weak $*$ topology to $l\big(X,T)$.
 
    In order to pass to the limit we pick some large $S>0$ and by standard manipulations on the
    integral (see the proof of Theorem~\ref{DPP} below), we have
    $$ \int_S^{+\infty} l\big(X_\e(s),T_\e(s)\big)\exp(-D_\e(s)) \ds = \exp(-D_\e(S))J(\tilde X,
    \tilde T,\tilde D,\tilde L)\; ,$$ 
    where $(\tilde X, \tilde T,\tilde D,\tilde L)$ is a trajectory starting from
    $(X_\e(S),T_\e(S),0,0)$ in $\cT(X_\e(S),T_\e(S))$. Hence $J(\tilde X, \tilde T,\tilde D,\tilde
    L)$ is bounded from below by a constant $\tilde K$ and we can rewrite the above property on
    $U(\xe,\te)$ as
     $$U(\xe,\te)\geq \int_0^{S} l\big(X_\e(s),T_\e(s)\big)\exp(-D_\e(s)) \ds+\tilde K\exp(-D_\e(S))
     -\e\; .$$

    We pass to the limit in this inequality and obtain
    $$\lim_\e U(\xe,\te) \geq \int_0^S l\big(X(s),T(s)\big)\exp(-D(s)) \ds+\tilde K\exp(-D(S))\; .$$
    Since this inequality is valid for any $S>0$, the arguments of the proof of Lemma~\ref{c-posit}
    implies that $s \mapsto l\big(X(s),T(s)\big)\exp(-D(s))$ is in $L^1(0,+\infty)$ \footnote{since
    the integrals on $E_1$ and $E_2$ are bounded and so only the integral on $E_3$ where the
    integrand is positive plays a real role in the $L^1$-property.} and letting $S \to +\infty$, we
    end up with 
    \begin{equation}\label{VF.lsc}
        \lim_\e U(\xe,\te) \geq \int_0^{+\infty} l\big(X(s),T(s)\big)\exp(-D(s)) \ds \geq
        U(x,t)\;,
    \end{equation} and the proof is complete.
    
     \smallskip

\noindent\textbf{(c)} Finally the existence of an optimal trajectory relies on exactly the same
arguments as above, \ie on the compactness of the trajectories.
\end{proof}

\section{Ishii solutions for the Bellman Equation}\label{IS-BE}

In this section we prove that the value function is a (discontinuous) viscosity solution of the
Bellman Equation associated with the control problem, namely
\begin{equation}\label{eq:bellmann.0}
    \F(x,t,u,Du)=0\quad\text{in}\quad\R^N\times [0,\Tf]\;,
\end{equation}
where, for any $x \in \R^N$, $t\in [0,\Tf]$, $r\in \R$ and $p=(p_x,p_t)\in \R^{N+1}$
\begin{equation}\label{def.F.global}
    \F(x,t,r,p):=\sup_{(b,c,l)\in\BCL(x,t)}\Big\{-b\cdot p+cr-l\big\}\;.
\end{equation}

Writing the Bellman Equation under the form \eqref{eq:bellmann.0} is a little bit formal: if a more or less
classical definition of viscosity sub and supersolutions can be used in $\R^N\times]0,\Tf]$
following Definition~\ref{class.visc. sol}, the case of $t=0$ requires a particular treatment.

Indeed, it is well-known that the supersolution inequality for such Bellman Equation is related to the optimality
of one or several trajectories while the subsolution one reflects the fact that any trajectory for
any possible control is sub-optimal. At a point $(x,0)$, the standard $\F\geq 0$ supersolution inequality does not seem
to cause any problem, even if the optimal trajectory has to stay on $\R^N\times\{0\}$. On the contrary, there is a
problem with the standard subsolution inequality since we cannot use any solution $(X,T,D,L)$ of the
$\BCL$-differential inclusion, but only those for which $b^t=0$. This is why the constraint to remain in $\R^N\times[0,\Tf]$ 
obliges us to change the definition of subsolution for $t=0$.

This leads to introduce the ``initial Hamiltonian''
\begin{equation}\label{def.F.init}
    \F_{init}(x,r,p_x):=\sup_{((b^x,0),c,l)\in\BCL(x,0)}\big\{-b^x \cdot p_x +cr - l \big\}\;.
\end{equation}

Before going further, we describe the properties of $\F$ and $\F_{init}$ in the following result.
\begin{lemma} 
    The Hamiltonians $(x,t,r,p)\mapsto\F(x,t,r,p)$ and $(x,t,r,p)\mapsto\F_{init}(x,r,p)$ are
    \usc with respect to all the variables, and convex and Lipschitz as a function of $r$ and $p$.
\end{lemma}
\begin{proof} We only provide the proof for $\F$, the one for $\F_{init}$ being analogous.

    For the upper semi-continuity, let us take a sequence $(x_n,t_n,r_n,p_n)\to
    (x,t,r,p)\in\R^N\times[0,\Tf]\times\R\times\R^{N+1}$. Since, for any $n$, 
    $\BCL(x_n,t_n)$ is compact, there exists $(b_n,c_n,l_n)\in\BCL(x_n,t_n)$ such that
    $$\F(x_n,t_n,r_n,p_n)= - b_n\cdot p_n +c_n r_n-l_n\;.$$
    Since $\BCL(\cdot,\cdot)$ is \usc as a set-valued map, it follows that, for any
    $\delta>0$, if $n$ is large enough, 
    $$(b_n,c_n,l_n)\in\BCL(x_n,t_n)\subset\BCL(x,t)+\delta B_{2N+3}\;,$$
    where $B_{2N+3}$ is the unit ball in $\R^{2N+3}$. For such $n$,
    $(b_n,c_n,l_n)$ can be decomposed as the sum
    $(\tilde b_n,\tilde c_n,\tilde l_n)+\delta e_n$ for some $(\tilde b_n,\tilde c_n,\tilde
    l_n)\in\BCL(x,t)$ and some $e_n\in B_{2N+3}$. Now, since $(x_n,t_n,r_n,p_n)$ is bounded,
    $$\begin{aligned}
        \F(x,t,r,p) &\geq -\tilde b_n\cdot p+\tilde c_n r -\tilde l_n\\
        &\geq -b_n\cdot p_n+c_n r_n-l_n - o_\delta(1)\\
        &\geq \F(x_n,t_n,r_n,p_n) - o_\delta(1)\;.
    \end{aligned}$$
    Passing to the limsup on $n$ and sending $\delta\to0$ yields the upper semi-continuity property.

    The Lipschitz continuity is just a consequence of the boundedness of the $b$ and $c$ components
    in $\BCL(x,t)$ for any $x$ and $t$\,: if $F(x,t,r,p)=-b\cdot p+cr-l$, then since
    $\F(x,t,r',q)\geq -b\cdot q+cr'-l$, we have $$\F(x,t,r,p)-\F(x,t,r',q)\leq |c||r-r'|+|b||p-q|\leq
    M\big(|r-r'|+|p-q|\big)\;,$$ and of course the converse inequality is also true.

    Finally, the convexity of $\F$ with respect to $(r,p)$ just comes from the fact that $\F$ is the
    supremum of affine functions in $(r,p)$. 
\end{proof}

\subsection{Discontinuous viscosity solutions}

Let us first give the definition based on the notion of discontinuous (or Ishii) viscosity solution
exposed in Chapter~\ref{chap:pde.tools}, but modified in a suitable way to take into account the
particularity of $t=0$.
\begin{definition}\label{def:sub.sup.gen} 
    A locally bounded function $u$ is a subsolution of \eqref{eq:bellmann.0} if its \usc enveloppe
    satisfies 
    \begin{equation}\label{eq:sub.H}
    \F_*(x,t,u^*, Du^*) \leq 0 \quad \hbox{on  } \R^N \times]0,\Tf]\; ,
    \end{equation}
    and, for $t=0$
    \begin{equation}\label{eq:sub.Ht0}
    \min(\F_*(x,0,u^*, Du^*), (\F_{init})_* (x,u^*(x,0),D_x u^*(x,0)))\leq 0 \quad\hbox{in  }\R^N \; .
    \end{equation}
    A locally bounded function $v$ is a supersolution \eqref{eq:bellmann.0} if its \lsc enveloppe
    satisfies 
    \begin{equation}\label{eq:super.H}
        \F(x,t,v_*, Dv_*) \geq 0 \quad \hbox{on  } \R^N \times[0,\Tf]\; .
    \end{equation}
    A locally bounded function is a viscosity solution of \eqref{eq:bellmann.0}
    if it is both a subsolution and a supersolution of
    \eqref{eq:bellmann.0}.
\end{definition}

For the supersolution property, the simple formulation comes from the fact that $\F$ is \usc in
$\R^N \times[0,\Tf] \times \R \times \R^N$. For the subsolution, the inequality is the expected one
 on $\R^N \times]0,\Tf]$ but is modified for $t=0$. In fact, we show below that the value function satisfies
 $$(\F_{init})_* (x,U^*(x,0),D_x U^*(x,0)))\leq 0 \quad\hbox{in  }\R^N\; ,$$
and Section~\ref{IC-HJB} (see Proposition~\ref{visc-ineq-init}) will confirm
that the $\F_*$-contribution  in \eqref{eq:sub.Ht0} is not necessary, the initial data condition 
being totally equivalent to $(\F_{init})_*\leq0$.


\subsection{The dynamic programming principle}

The first step towards establishing the sub/supersolution properties of $U$ is to prove the
classical 
\index{Dynamic Programming Principle!simplest form}
\begin{theorem}\label{DPP}\emph{--- Dynamic Programming Principle.} \smsp
Under hypothesis \HBCL, the value function $U$ satisfies
$$
    U(x,t)=
    \inf_{\cT(x,t)}\Big\{\int_0^\theta
    l\big(X(s),T(s)\big)\exp(-D(s)) \ds+U \big(X(\theta),T(\theta))\exp(-D(\theta))\Big\}\;,
$$
for any $(x,t)\in\R^N\times(0,\Tf]$, $\theta >0$.
\end{theorem}

\begin{proof}
    Let us denote by $J_\theta(X,T,D,L)$ the integral over $(0,\theta)$ inside the $\inf$ and by
    $\hat U(x,t)$ the complete right-hand side, while $U(x,t)=\inf_{\cT(x,t)}J(X,T,D,L)$ and
    $J(\cdot)$ stands for the integral over $(0,+\infty)$. 

    \smallskip

    \noindent\textbf{(a)} Let us prove that $U\leq\hat U$.
    We first take any trajectory $(X,T,D,L)\in\cT(x,t)$. Then, noting 
    $(x_\theta,t_\theta):=(X(\theta),T(\theta))$, we select an $\eps$-optimal trajectory
    $(X^\e,T^\e,D^\e,L^\e)\in\cT(x_\theta,t_\theta)$, in the sense that
    $$U(x_\theta,t_\theta)\leq J(X^\e,T^\e,D^\e,L^\e)+\eps\;.$$
    We then construct a new trajectory in $\cT(x,t)$ by setting 
    $$(\hat X,\hat T,\hat D,\hat L)(s):=
    \begin{cases}
        (X,T,D,L)(s) & \text{ if } 0\leq s\leq \theta\;,\\
        (X^\e,T^\e,D^\e+D(\theta),L^\e+L(\theta))(s-\theta) & \text{ if }s>\theta\;.
    \end{cases}$$
    Using the definition of $U(x,t)$ we get
    $$\begin{aligned}
        U(x,t) &\leq  J(\hat X,\hat T,\hat D,\hat L)\\ 
         &\leq J_\theta(X,T,D,L)+\int_\theta^{+\infty} 
        l(X^\e(s-\theta),T^\e(s-\theta))\exp(- D^\e(s-\theta)-D(\theta))\ds\\
        &\leq J_\theta(X,T,D,L)+\exp(-D(\theta))\int_0^{+\infty}l(X^\e(s),T^\e(s))
        \exp(- D^\e(s))\ds\\
        & \leq J_\theta(X,T,D,L)+\exp(-D(\theta))(U(X(\theta),T(\theta))+\eps)
    \end{aligned}$$
    Notice that the trajectory $(X,T,D,L)\in\cT(x,t)$ is arbitrary and does not depend on $\eps$.
    Hence, taking the infimum over $\cT(x,t)$ and sending $\eps$ to zero, we conclude that indeed
    $U\leq \hat U$.

    \smallskip

    \noindent\textbf{(b)} The converse inequality follows from similar manipulations: let us take an
    $\e$-optimal trajectory $(X^\e,T^\e,D^\e,L^\e)\in\cT(x,t)$ for estimating $U(x,t)$. 
    After separating the integral in
    two parts and changing variable $s\mapsto s-\theta$ in the second part we get
    \begin{equation}\label{ineq.vf.hatu}
    \begin{aligned}
        U(x,t)+\eps \geq & \ J_\theta(X^\e,T^\e,D^\e,L^\e)\\ & \ + \int_0^{+\infty}
        l(X^\e(s+\theta),T^\e(s+\theta))\exp(-D^\e(s+\theta))\ds\;.
    \end{aligned}
    \end{equation}
    The trajectory 
    $(X,T,D,L)(s):=(X^\e,T^\e,D^\e,L^\e)(s+\theta)-(0,0,D^\e(\theta),L^\e(\theta))$ belongs
    to $\cT(X^\e(\theta),T^\e(\theta))$, and \eqref{ineq.vf.hatu} can be written as
    \begin{align*}
        U(x,t)+\eps \geq &\  J_\theta(X^\e,T^\e,D^\e,L^\e)\\ 
        & \ + \exp(-D^\e(\theta))\int_0^{+\infty}
        l(X(s),T(s))\exp(-D(s))\ds\,;\\[2mm]
         \geq &\ J_\theta(X^\e,T^\e,D^\e,L^\e) + \exp(-D^\e(\theta))J(X,T,D,L)\;.
    \end{align*}
    Now, using $(X,T,D,L)$ as an admissible trajectory starting at $(X^\e(\theta),T^\e(\theta))$
    we use the estimate 
    $$U(X^\e(\theta),T^\e(\theta))\leq J(X,T,D,L)\,$$
    to get the inequality
    $$U(x,t)+\eps\geq
    J_\theta(X^\e,T^\e,D^\e,L^\e)+\exp(-D^\e(\theta))U(X^\e(\theta),T^\e(\theta))\;.$$
    Finally, $\hat U(x,t)$ being the infimum of all trajectories in $\cT(x,t)$, 
    the right-hand side is greater than or equal to $\hat U(x,t)$ and the conclusion follows.
\end{proof}


\subsection{The value function is an Ishii solution}
\label{subsec:vf.ishii}

Following the definition recalled in Section~\ref{IS-BE}, we first prove the 
Following Definition~\ref{def:sub.sup.gen}, we first prove the 
\begin{theorem}\label{SP}\emph{--- Supersolution Property.}\smsp 
    Under assumption \HBCL, the value function $U$ is a viscosity supersolution of the Bellman
    equation~\eqref{eq:bellmann.0}.
\end{theorem}

\begin{proof} 
    We keep here the notation $J_\theta(X,T,D,L)$ introduced in the proof of Proposition~\ref{DPP}
    for the integral over $(0,\theta)$ in the dynamic programming principle. 

    \smallskip
    
    In this proof, we are going to ignore on purpose that we know that $U$ is \lsc on $\R^N\times
    [0,\Tf]$. Therefore, we are going to actually prove that $U_*$ is a supersolution. The reason
    to do so is to show that the proof of this property is robust and does not require a priori
    the information that $U$ is \lsc

    Let $(x,t)\in\R^N\times[0,\Tf]$ be a local minimum point of $U_*-\phi$ where $\phi\in
    C^1(\R^N\times[0,\Tf])$. We can assume without loss of generality that $U_*(x,t)=\phi(x,t)$. In
    particular, $U\geq U_*\geq\phi$ in a neighborhood of $(x,t)$.  Moreover, by definition of the
    lower semi-continuous envelope, there exists a sequence $(x_n,t_n)\to(x,t)$ such that
    $U(x_n,t_n)\to  U_*(x,t)$.
    
    We apply the dynamic programming principle for $U$ at $(x_n,t_n)$: 
    $$U(x_n,t_n)= \inf_{\cT(x_n,t_n)}\Big(J_\theta(X_n,T_n,D_n,L_n) +
    U(X_n(\theta),T_n(\theta))\exp(-D_n(\theta))\Big)\;.$$ 
    On one hand, for the left-hand side, using the definition of the sequence $(x_n,t_n)$, the fact
    that $U_*(x,t)=\phi(x,t)$ and the continuity of $\phi$, there exists a sequence $(\e_n)_n$ of
    non-negative real numbers converging to $0$ such that $U(x_n,t_n)\leq \phi(x_n,t_n)+\e_n$.

    On the other hand, since $|b|\leq M$ is bounded, if $\theta$ is small enough the trajectory
    $(X_n(s),T_n(s))$ remains close enough to $(x,t)$ and we can use the inequalities $U\geq
    U_*\geq\phi$, the last one coming from the local minimum point property. This yields
    \begin{equation}\label{super.prog.dyn.phi}
        \phi(x_n,t_n)+\e_n
        \geq \inf_{\cT(x_n,t_n)}\Big( J_\theta(X_n,T_n,D_n,L_n)
        + \phi(X_n(\theta),T_n(\theta))\exp(-D_n(\theta))\Big)\;.
    \end{equation}

    For simplicity of notations, we set $Z_s:=(X_n(s),T_n(s))$.  Since $\phi$ is $C^1$, the
    following expansion holds 
    \begin{equation}\label{expansion.phi}
        \begin{aligned}
        \phi(Z_\theta)\exp(-D(\theta))-\phi(Z_0) &=
        \int_0^\theta\frac{d}{ds}\Big(\phi(Z_s)\exp(-D_n(s))\Big)\ds \\
        &= \int_0^\theta \Big(b(Z_s)\cdot
        D\phi(Z_s)-c(Z_s)\phi(Z_s)\Big)\exp(-D_n(s))\ds\;.
     \end{aligned}\end{equation}
     Combining with \eqref{super.prog.dyn.phi} yields
    $$\begin{aligned}
        0\geq &\inf_{\cT(x_n,t_n)} \int_0^\theta \Big\{ b(Z_s)\cdot D\phi(Z_s) -c(Z_s)\phi(Z_s)+
        l(Z_s)\Big\}\exp(-D_n(s))\ds-\e_n\;,\\ \geq & \int_0^\theta
        -\F(X_n(s),T_n(s),\phi(X_n(s),T_n(s)),D\phi(X_n(s),T_n(s)))\exp(-D_n(s))\ds -\e_n\;.
    \end{aligned}$$
    Since $\theta$ is arbitrary, we can choose a sequence $\theta_n$ in order that
    $\e_n\theta_n^{-1}\to 0$.  We remark that
    $(X_n(s),T_n(s),\phi(X_n(s),T_n(s)),D\phi(X_n(s),T_n(s)) \to (x,t, \phi(x,t),D\phi(x,t))$.
    Therefore, if $\delta >0$ is fixed and small, provided $n$ is large enough we have 
    $$\F(X_n(s),T_n(s),\phi(X_n(s),T_n(s)),D\phi(X_n(s),T_n(s)) \leq 
    \F(x,t, \phi(x,t),D\phi(x,t)) +\delta\; .$$ 
    In addition, $\exp(-D_n(s))=1+O(\theta_n)$; so,
    using all these informations in the above inequality, we deduce that 
    $$ 0\geq \theta_n\left( -\F(x,t,\phi(x,t),D\phi(x,t))-\delta (1+O(\theta_n)) \right)-\e_n\; .$$
    Dividing by $\theta_n$ and letting $n$ tend to infinity, we obtain
    $\F(x,t,\phi(x,t),D\phi(x,t))+\delta \geq0$ and this inequality being true for any
    $\delta>0$, the result is proved.
\end{proof}

Now we turn to the subsolution properties and to do so, we first need a result for the \usc
enveloppe of $U$ at $t=0$:
\begin{lemma}\label{VF-usc-init}
    Under assumption \HBCL, we have, for any $x\in \R^N$
    $$ U^*(x,0) = \limsup_{y\to x} U(y,0)\; .$$
    In other word, the \usc enveloppe of $U$ at points $(x,0)$ can be computed by using only $U$ on
    $\R^N\times\{0\}$.  
\end{lemma}
\begin{proof}
    By definition of $U^*$, there exists a sequence $(x_\e,t_\e) \to (x,0)$ such
    that  $U(x_\e,t_\e) \to U^*(x,0)$. Then we apply the dynamic programming principle
    $$U(x_\e,t_\e)= \inf_{\cT(x_\e,t_\e)}\Big(J_\theta(X_\e,T_\e,D_\e,L_\e) +
    U(X_\e(\theta),T_\e(\theta))\exp(-D_\e(\theta))\Big)\;.$$
    We consider a trajectory $(X_\e,T_\e,D_\e,L_\e)$ which is solution of the differential inclusion
    associated with $\BCL_\flat$ defined in the proof of Lemma~\ref{u-bound}, \ie with $b^t=-1$,
    and we use it in the dynamic programming principle with $\theta=t_\e$. Since
    $T_\e(\theta)=T_\e(t_\e)=0$, we obtain
    $$U(x_\e,t_\e)\leq J_{t_\e}(X_\e,T_\e,D_\e,L_\e) + U(X_\e(t_\e),0)\exp(-D_\e(t_\e))\;.$$
    But $J_{t_\e}(X_\e,T_\e,D_\e,L_\e)=O(\te)$ and $\exp(-D_\e(t_\e))=1 +O(t_\e)$, therefore:
    $$ U(x_\e,t_\e) \leq U(X_\e(t_\e),0) +O(t_\e)\; ,$$
    and $U^*(x,0)\leq  \limsup U(x_\e,t_\e)\leq \limsup U(X_\e(t_\e),0)\leq U^*(x,0)$, 
    proving the claim.
\end{proof}

Now we can prove the 
\begin{theorem}\label{thm:SubP}\emph{--- Subsolution Properties.}\smsp 
    Under assumption \HBCL, the value function $U$ is a viscosity subsolution of 
    \begin{equation}\label{sub.vf}
        \F_*(x,t,U, DU) \leq 0 \quad \hbox{on  } \R^N \times]0,\Tf]\; ,
    \end{equation}
    and for $t=0$, it satisfies
    \begin{equation}\label{sub.init.vf}
       (\F_{init})_* (x,U(x,0),D_x U(x,0))\leq 0 \quad\hbox{in  }\R^N,
    \end{equation}
    hence it is a subsolution of \eqref{eq:bellmann.0}.
\end{theorem}

\begin{proof} 
    The proof is more involved than for the supersolution condition, first because we need to
    consider $\F_*$ which a priori differs from $\F$, but also because we face here the potential
    discontinuities of $b,c,l$ with respect to $x,t$. 

    \

    \noindent\textbf{(a)}
    We first prove \eqref{sub.vf}. We consider a maximum point $(x,t)\in\R^N\times]0,\Tf]$ of
    $U^*-\phi$ where $\phi$ is a $C^1$ test-function and, as above, we assume that
    $U^*(x,t)=\phi(x,t)$. By definition of $U^*$,
    there exists a sequence $(x_n,t_n)\to(x,t)$ such that $U(x_n,t_n)\to U^*(x,t)$ and, by the
    continuity of $\phi$, we also have $\phi(x_n,t_n)\leq U(x_n,t_n)+\e_n$ for some sequence
    $(\e_n)_n$ of non-negative real numbers converging to $0$.
 
    Applying the dynamic programming principle, we have 
    $$ U(x_n,t_n)=\inf_{\cT(x_n,t_n)} \Big(J_\theta(X,T,D,L))+
    U(X(\theta),T(\theta))\exp(-D(\theta))\Big)\;.$$ 
    If $\theta>0$ is small enough, the maximum point property implies 
    $$U(X(\theta),T(\theta))\leq U^*(X(\theta),T(\theta))\leq\phi(X(\theta),T(\theta))$$
    and therefore we obtain
    $$\phi(x_n,t_n)-\e_n \leq \inf_{\cT(x_n,t_n)}
    \Big(J_\theta(X,T,D,L))+ \phi(X(\theta),T(\theta))\exp(-D(\theta))\Big)\;.$$ 
    Using expansion \eqref{expansion.phi}---here also with the notation $Z_s=(X(s),T(s))$---leads to
    \begin{equation}\label{ineq.theta.n}
    \int_0^\theta \Big(-b(Z_s)D\phi(Z_s)+c(Z_s)\phi(Z_s)-l(Z_s)\Big)
    \exp(-D(s))\ds\leq\e_n\;,
    \end{equation}
    for any trajectory $(X,T,D,L)\in\cT(x_n,t_n)$.

    In order to conclude, we have to show that, for any $n$, we can choose a trajectory
    $(X,T,D,L)_n\in\cT(x_n,t_n)$ such that the integral is close to
    $\F_*(x,t,\phi(x,t),D\phi(x,t))$.

    \

    \noindent\textbf{(b)}
    To do so, we are going to solve a suitable differential inclusion for a set-valued map that we
    build in the following way.  We  consider the auxiliary function $h_\phi(b,c,l):=-b\cdot
    D\phi(x,t)+c\phi(x,t)-l$ and for $\delta>0$, we define a restricted set-valued map for $(y,s)$
    in a neighborhood of $(x,t)$ as follows
    $$
    \BCLloc(y,s):=\BCL(y,s)\cap \Big\{h_\phi(b,c,l)\geq
    \F_*(x,t,\phi(x,t),D\phi(x,t))-\delta\,\Big\}\;.
    $$
    We claim that $\BCLloc$ is not empty and satisfies \HBCLa, at least for $(y,s)$ close enough to
    $(x,t)$.
    
    Indeed, if on the contrary, $\BCLloc (y_n,s_n)$ is empty for some sequence $(y_n,s_n)\to(x,t)$,
    this means that, for any $(b,c,l)\in\BCL(y_n,s_n)$, we have  
    $$h_\phi(b,c,l)<\F_*(x,t,\phi(x,t),D\phi(x,t))-\delta\;,$$
    which implies that 
    $$\F(y_n,s_n,\phi(x,t),D\phi(x,t))=\sup_{(b,c,l)\in\BCL(y_n,s_n)}h_\phi(b,c,l)\leq
    \F_*(x,t,\phi(x,t),D\phi(x,t))-\delta\;.$$
    But using the lower semi-continuity of $\F_*$ we are led to a contradiction since
    $$\begin{aligned}
        \F_*(x,t,\phi(x,t),D\phi(x,t)) &\leq \liminf_{n\to\infty}
        \F(y_n,s_n,\phi(x,t),D\phi(x,t))\\
        & \leq\F_*(x,t,\phi(x,t),D\phi(x,t))-\delta\;.
    \end{aligned}$$
    Concerning the images $\BCLloc(y,s)$, they are clearly convex and compact from the properties
    of $\BCL$ and the fact that the set $\{h_\phi\geq \alpha\}$ is closed and convex.  Moreover, the
    \usc property derives from the fact that $\BCL$ is \usc while $\{h_\phi\geq\delta\}$ is a fixed
    set. 

    \

    \noindent\textbf{(c)}
    Hence we can solve the differential inclusion associated to $\BCLloc\subset\BCL$ with initial
    data $(x_n,t_n)$ on a small time interval $(0,\theta)$. For this specific trajectory, up to
    taking $\theta$ smaller and $n$ larger, using that $\phi$ is $C^1$, we get for $s\in[0,\theta]$
    $$\begin{aligned}
      -b(Z_s)\cdot D\phi(Z_s)+c(Z_s)\phi(Z_s) -l(Z_s)&= h_\phi(b(Z_s),c(Z_s),l(Z_s)) + O(\theta)\\
        & \geq \F_*(x,t,\phi(x,t),D\phi(x,t))-\delta + O(\theta)\;.
    \end{aligned}$$
    Plugging this into \eqref{ineq.theta.n}, using also that $\exp(-D(s)) = 1 + O(\theta)$, we get
    $$
    \theta\bigl( \F_*(x,t,\phi(x,t),D\phi(x,t))-\delta+ O(\theta)\bigr)(1 + O(\theta))\leq \e_n\;.
    $$
    To conclude, we send $n\to\infty$ and then we divide by $\theta$ and we send it to $0$.  We end
    up with the inequality $\F_*(x,t,\phi(x,t),D\phi(x,t))\leq \delta$ for any $\delta>0$ and
    therefore $\F_*(x,t,\phi(x,t),D\phi(x,t))\leq 0$.
    
    \

    \noindent\textbf{(c)}
    Now we turn to \eqref{sub.init.vf}, which is treated by the same technique as above, using 
    Lemma~\ref{VF-usc-init}: the same proof as above readily applied since we can choose $t_n=0$ and
    therefore, in the definition of $\BCLloc$, we can consider only the $b$ such that $b^t=0$ and
    replace $\F_*$ by $(\F_{init})_*$. Indeed, any relevant trajectory starting from $(x_n,0)$
    necessarily satisfies $b^t(Z_s)=0$.
\end{proof}

As we shall see later on in this book, Ishii solutions are not unique in general in the
presence of discontinuities. Nevertheless, we prove below that $U$ is the minimal one, see
Corollary~\ref{VFm-minsup}, and we will explain later on several ways in which we can recover some
uniqueness.

\section{Supersolutions of the Bellman Equation}

\subsection{The super-dynamic programming principle}

We prove here that supersolutions always satisfy a super-dynamic
programming principle. Again, we remark that this result is independent of the possible
discontinuities for the dynamic, discount factor and cost. But to prove it, we need an additional
ingredient in which we assume that we have already used Lemma~\ref{c-posit} to reduce to
the case when $c\geq 0$.
\begin{lemma}\label{subsol-chi}
    Assume \HBCLb holds and let $\chi(t)=-K (t+1)$ for $K>0$ large enough. Then, for
    any $(x,t) \in \R^N\times [0,\Tf]$ and any  $(b,c,l)\in \BCL(x,t)$,
    $$ - b\cdot D\chi(t)+c\chi(t)-l \leq - \uc  < 0 \;.$$
\end{lemma}
\begin{proof} This is just obtained by direct computation: $-b\cdot D\chi(t)=Kb^t\leq0$ 
    while $c\chi(t)-l\leq-Kc-l$. By taking $K\geq (\uc+l)/c$, we get the result.
\end{proof}

Lemma~\ref{subsol-chi}, which is valid both for $t>0$ and $t=0$, provides a very classical
property: the underlying HJB Equation has a strict subsolution, which is a key
point in comparison results. Of course, in this time-dependent case, one could say
that such property is obvious. But we are not completely in a standard time-dependent case since we
recall that $b^t=0$ is allowed potentially for any $t\geq0$.

Our next result is the\index{Dynamic Programming Principle!for supersolutions}
\begin{lemma}\label{lem:super.dpp}
    Under assumption \HBCL, if $v$ is a bounded \lsc supersolution of \eqref{eq:super.H}  in
    $\R^N\times(0,\Tf]$, then, for any $(\xb,\tb)\in\R^N\times(0,\Tf]$ and any $\sigma >0$,
    \begin{equation}\label{ineq:super.dpp}
        v(\xb,\tb)\geq
        \inf_{\cT(\xb,\tb)}
        \Big\{\int_0^{\sigma}
        l\big(X(s),T(s)\big)\exp(-D(s)) \ds+v\big(X(\sigma),T(\sigma)\big)\exp(-D(\sigma))\Big\}
    \end{equation}
\end{lemma}

\begin{proof} 
    To begin with, because of Lemma~\ref{c-posit} we can assume that $c\geq 0$ for any $(b,c,l) \in
    \BCL(x,t)$ and for any $(x,t)$.  Fixing $(\xb,\tb)$ and $\sigma>0$, we argue through a
    three-step proof involving a regularization procedure and comparison result in the compact
    domain $$\Kxt:=\overline{B(\xb,M\sigma)}\times [0,\tb]\;,$$ 
    where $M$ is given by \HBCLa. 

    \medskip

    \noindent\textsc{Step 1: regularization ---}
    We consider a sequence of regularized Hamiltonians using the penalization function 
    $$\psi(b,c,l,x,t)=\inf_{(y,s) \in \R^N\times [0,\Tf]} 
    \Big(\dist\big((b,c,l),\BCL(y,s)\big)+ |y-x|+|t-s|\Big) \; , $$
    where $\dist(\cdot,\BCL(y,s))$ denotes the distance to the set $\BCL(y,s)$.
    We notice that $\psi$ is Lipschitz continuous and that $\psi(b,c,l,x,t)=0$ if
    $(b,c,l)\in\BCL(x,t)$. Then we set 
    $$\F_\delta(x,t,r,p):=\sup_{(b_\delta ,c_\delta,l_\delta )\in \BCL_\delta (x,t)}
    \big\{ -b_\delta \cdot p +c_\delta  r -l_\delta  \big\}  \;,$$
    where $\BCL_\delta (x,t)$ is the set of all $(b_\delta ,c_\delta ,l_\delta)
    \in\R^{N+1}\times\R\times\R$ such that $|b_\delta^x|\leq M$,
    $-1 \leq b_\delta ^t\leq 0 $, $0\leq c_\delta \leq M$ 
    and 
    $$l_\delta= l + \delta^{-1}\psi\Big(b_\delta ,c_\delta ,l, x,t\Big)\quad
    \text{for some } |l| \leq M\;.$$
    This sequence of Hamiltonians enjoys the following straightforward properties:
    \begin{enumerate}
        \item[$(i)$] for any $\delta >0$, $\F_\delta\geq \F $ and therefore $v$ is a \lsc supersolution
    of $\F_\delta\geq 0$\\ on $B(\xb,M\sigma) \times (0,t]$;
    \item[$(ii)$] the Hamiltonians $\F_\delta$  are (globally) Lipschitz continuous w.r.t.  all variables;
    \item[$(iii)$] $\F_\delta \downarrow \F$ as $\delta \to 0$, all the other variables being fixed.
    \end{enumerate}

    \smallskip

    On the other hand, $v$ being \lsc on $\Kxt$, there exists an increasing sequence
    $(v_\delta)_\delta$ of Lipschitz continuous functions such that $v_\delta \leq v$ and
    $\sup_\delta v_\delta =v$ on $\Kxt$.

    For $(x,t)\in \Kxt$, we now introduce the function
    $$\begin{aligned}
	u_\delta(x,t):=
        \inf \Big\{\int_0^{\sigma \wedge \theta} &
        l_\delta\big(X_\delta(s),T_\delta(s)\big)\exp(-D_\delta (s)) \ds\\
		& + v_\delta \big(X_\delta(\sigma \wedge \theta),T_\delta(\sigma \wedge \theta)\big)\exp(-
        D_\delta(\sigma \wedge \theta))\Big\}\; ,
    \end{aligned}$$
    where $(X_\delta,T_\delta, D_\delta,L_\delta)$ is a solution of the differential inclusion 
    $$\begin{cases}
        & (\dot X_\delta,\dot T_\delta,\dot D_\delta , \dot L_\delta ) (s) \in \BCL_\delta
        (X_\delta(s),T_\delta(s))\;,\\
        & (X_\delta,T_\delta, D_\delta,L_\delta)(0)=(x,t,0,0)\;,
    \end{cases}$$
    the infimum being taken over all trajectories $X_\delta$ which stay in $\overline{B(\xb,M\sigma)}$ 
    till time $\sigma \wedge \theta$ and any stopping time $\theta$ such that either 
    $X_\delta(\theta)$ on $\partial B(\xb,M\sigma)$ or $T_\delta(\theta)=0$. 

    By classical arguments, $u_\delta$ is continuous since all the data involved are continuous,
    $u_\delta\leq v_\delta$ on $(\partial B(\xb,M\sigma) \times [0,\tb])\cup(B(\xb,M\sigma) \times
    \{0\})$ (for the same reason) and $u_\delta$ satisfies 
    $$ \F_\delta(x,t,u,Du)= 0 \quad \hbox{in }B(\xb,M\sigma) \times (0,\tb]\; .$$ 
    Notice that this equation and the one for $v_\delta$ hold up to time $\tb$, as a consequence of
    the fact that $b^t\leq 0$ for all $b\in \B(x,t)$ and all $(x,t)$.

    \medskip

    \noindent\textsc{Step 2: comparison for the approximated problem ---}
    In order to show that $u_\delta\leq v$ in $\Kxt$ we argue by contradiction assuming that
    $\max_{\Kxt}(u_\delta - v) >0$.

    We consider the function $\chi$ given by Lemma~\ref{subsol-chi}: using the definition of
    $l_\delta$, it is easy to show that 
    $$ \F_\delta(x,t,\chi,D\chi)\leq -\uc < 0 \quad \hbox{in }B(\xb,M\sigma) \times (0,\tb]\; ,$$
    and, by convexity, for any $0<\mu<1$, $u_{\delta,\mu}=\mu u_\delta +(1-\mu)\chi$ is a
    subsolution of 
    $$ \F_\delta(x,t,u_{\delta,\mu},Du_{\delta,\mu})\leq -(1-\mu)\uc < 0 \quad \hbox{in
    }B(\xb,M\sigma) \times (0,\tb]\; .$$

    Moreover, if $\mu<1$ is close enough to $1$, we still have $\max_{\Kxt}(u_{\delta,\mu} - v) >0$
    and we can choose $K$ large enough in order to have $u_{\delta,\mu} \leq v_\delta$ on $(\partial
    B(\xb,M\sigma) \times [0,\tb])\cup(B(\xb,M\sigma)
    \times \{0\}$.

    If $(\tilde x,\tilde t)\in \Kxt$ is a maximum point of
    $u_{\delta,\mu} - v$, we remark that $(\tilde x,\tilde t)$ cannot be on $(\partial
    B(\xb,M\sigma) \times [0,\tb])\cup(B(\xb,M\sigma) \times \{0\})$ since on these parts of the
    boundary $u_{\delta,\mu} \leq v$. 

    Now we perform the standard proof using the doubling of variables with the test-function 
    $$u_{\delta,\mu} (x,t)-v(y,s)-\frac{|x-y|^2}{\epsilon^2}-\frac{|t-s|^2}{\epsilon^2}-
    (x-\tilde x)^2-(t-\tilde t)^2\;.$$
    By standard arguments, see Lemma~\ref{lem:cv-pen}, this function has a maximum point
    $(\xe,\te,\ye,\se)$ which converges to $(\tilde x,\tilde t,\tilde x,\tilde t)$ since $(\tilde
    x,\tilde t)$ is a strict global maximum point of  $(y,s) \mapsto u_{\delta,\mu}
    (y,s)-v(y,s)-(y-\tilde x)^2-(s-\tb)^2$ in $\Kxt$.

    We use now the $ \F_\delta$-supersolution inequality for $v$, the strict subsolution inequality
    for $ u_{\delta,\mu}$ and the regularity of $\F_\delta$ together with the fact that $c\geq 0$
    for all $(b,c,l)\in \BCL(y,s)$ [or $\BCL_\delta (y,s)$] and any $(y,s) \in
    \Kxt$. We are led to the inequality $$o(1) \leq
    -(1-\mu)\exp(-K\tb) \eta<0\;,$$ which yields a contradiction. Sending $\mu\to1$, we get that
    $u_\delta\leq v$ in $\Kxt$.

    \medskip

    \noindent\textsc{Step 3: passing to the limit ---} 
    To conclude the proof, we use the inequality $u_\delta (\xb,\tb) \leq
    v(\xb,\tb)$ and we first remark that, in the definition of $u_\delta (\xb,\tb)$, necessarily
    $\sigma \wedge\theta=\sigma $ since the trajectory $X_\delta$ cannot exit $B(\xb,M\sigma)$
    before time $\sigma$. Then, in order to let $\delta$ tend to $0$ in this inequality,
    we pick a $\delta$-optimal trajectory $(X_\delta,T_\delta, D_\delta,L_\delta)$.

    By the uniform bounds on  $(\dot X_\delta,\dot T_\delta, \dot D_\delta,\dot  L_\delta),$
    Ascoli-Arzela's Theorem implies that up to the extraction of a subsequence, we may assume that
    $(X_\delta,T_\delta, D_\delta,L_\delta)$ converges locally uniformly on $[0,+\infty)$ 
    to some $(X,T,D,L)$. We may also assume that their derivatives converge in $L^\infty$ weak-*
    topology (in particular $\dot L_\delta =l_\delta$).

    Using the $\delta$-optimal trajectory for approching $u_\delta$ leads to
    \begin{equation}\label{spdd}
        \begin{aligned}
            v(\xb,\tb) \geq 
            \int_0^{\sigma} & l_\delta\big(X_\delta(s),T_\delta(s)\big)\exp(-D_\delta (s)) \ds\\
            & +v_\delta \big(X_\delta(\sigma),T_\delta(\sigma)\big)\exp(-D_\delta(\sigma))
            -\delta\;,
        \end{aligned}
    \end{equation}
    an inequality that we use in two ways. 

    First, by multiplying by $\delta$ and using that $v$ and $v_\delta$ are bounded. 
    Writing $Z_s=(X_\delta(s),T_\delta(s))$ for simplicity, we obtain
    $$
        \int_0^{\sigma} \psi\Big(b_\delta (Z_s),c_\delta (Z_s),l_\delta (z_s),
        X_\delta(s),T_\delta(s)\Big)\exp(-D_\delta (s))ds = O(\delta)\;.
    $$
    By classical results on weak convergence, since the functions $(b_\delta,c_\delta,l_\delta)$
    converge weakly to $(b,c,l)$, there exists $\mu_s \in L^\infty \big(0,t;\mathbb{P}(B(0,M)\times
    [-M,M]^2\big))$ where $\mathbb{P}(B(0,M)\times [-M,M]^2)$ is the set of probability measures on
    $B(0,M)\times [-M,M]^2$ such that, taking into account the uniform convergence of $X_\delta,
    T_\delta$ and $D_\delta$, we have
    $$
    \int_0^{\sigma} \int_{B(0,M)\times [-M,M]^2} \psi\Big(b,c,l,
    X(s),T(s)\Big)\exp(-D (s))\d\mu_s(b,c,l) \ds = $$
    $$\lim_{\delta\to0} \int_0^{\sigma} \psi\Big(b_\delta (s),c_\delta (s),l_\delta (s),
    X_\delta(s),T_\delta(s)\Big)\exp(-D_\delta (s))ds = 0\; .$$
    We remark that $\psi \geq 0$ and $\psi(b,c,l,x,t)=0$ if and only if $(b,c,l)\in
    \BCL(x,t)$, therefore $(X,T,D,L)$ is a solution of the $\BCL$-differential inclusion.

    Second, we come back to \eqref{spdd} after recalling that $\psi$ is nonnegative, which implies that
    $l_\delta\big(X_\delta (s),T_\delta(s)\big) \geq l\big(X_\delta (s),T_\delta(s)\big)$ and 
    therefore
    $$
        \int_0^{\sigma} l \big(X_\delta(s),T_\delta(s)\big) \exp(-D_\delta (s)) \ds+v_\delta
        \big(X_\delta(\sigma),T_\delta(\sigma)\big)\exp(-D_\delta(\sigma)) -\delta \leq v(x,t)\;.
    $$
    We pass to the limit in this inequality using the lower-semicontinuity of $v$, together with
    the uniform convergence of $X_\delta,T_\delta,D_\delta$ and the dominated convergence theorem
    for the $l$-term. In particular,
    $$
    \liminf_{\delta\to0} \Big(v_\delta \big(X_\delta(\sigma),T_\delta(\sigma)\big)\Big)\geq
    v\big(X(\sigma),T(\sigma)\big)\;,
    $$
    which yields 
    $$
    \int_0^{\sigma} l(X(s),T(s)) \exp(-D (s))\ds+v\big(X(\sigma),T(\sigma)\big)\exp(-D(\sigma)) \leq
    v(\xb,\tb)\; .
    $$
    Finally, recalling that $(X,T,D,L)$ is a solution of the $\BCL$-differential inclusion, taking
    the infimum in the left-hand side over all solutions of this differential inclusion gives the
    desired inequality.
\end{proof}

\subsection{The value function is the minimal supersolution}

An easy consequence of Lemma~\ref{lem:super.dpp} is the
\begin{corollary}\label{VFm-minsup}\emph{--- Minimality of the value function.}\smsp
    Under assumption \HBCL, the value function $U$ is the minimal Ishii supersolution of
    \eqref{eq:super.H}.  
\end{corollary}

\begin{proof}
    Let $v$ be any bounded \lsc supersolution in the Ishii sense of $\F=0$. Using \eqref{ineq:super.dpp} we
    see that for any $(x,t)\in\R^N\times(0,\Tf]$ and $\sigma>0$,
    $$
    v(x,t)\geq  \inf_{\cT(x,t)}  \Big\{\int_0^{\sigma}
        l\big(X(s),T(s)\big)\exp(-D(s)) \ds+v\big(X(\sigma),T(\sigma)\big)\exp(-D(\sigma))\Big\}\;.
    $$
    Sending $\sigma\to+\infty$, we see that in particular for any trajectory $(X,T,D,L)$ 
    the integral $J(X,T,D,L)$ in Lemma~\ref{u-bound}-$(ii)$ is bounded by $2\Vert
    v\Vert_\infty$.

    Therefore, $D(\sigma)\to0$ as $\sigma\to +\infty$ and passing to the limit in the dynamic
    programming principle yields
    $$
    v(x,t)\geq  \inf_{\cT(x,t)}  \int_0^{+\infty}
        l\big(X(s),T(s)\big)\exp(-D(s)) \ds = U(x,t)\;.
    $$
    The conclusion is that $v\geq U$, which proves the minimality of the value function.
\end{proof}

We end this chapter by some comment: as we saw, the situation is not totally symmetric between
general Ishii supersolutions and subsolutions. For supersolutions, properties derive directly from
the Bellman Equation while the treatment of general subsolutions requires more advanced tools and
some structure assumption on the discontinuities. This is done in Chapter~\ref{chap:mixed.tools}.

\chapter{Mixed Tools}
\fancyhead[CO]{HJ-Equations with Discontinuities: Mixed Tools}
\label{chap:mixed.tools}

\abstract{This chapter contains all the results either connecting the optimal control problem and
the associated HJB Equation, or using both of them simultaneously. Included here are four very important
building blocks, in particular for the study of stratified problems:
$(i)$ a general formulation for the initial condition which gives the way to compute the initial data;
$(ii)$ the dynamic programming principle for subsolutions, a key ingredient in the proof of the comparison result
for stratified problems;
$(iii)$ the ``Magical Lemma'', which gives the local comparison argument for stratified problems;
$(iv)$ the description of the ``good assumptions'' needed for stratified problems.
}

\section{Initial conditions for sub and supersolutions of the Bellman Equation}
\label{IC-HJB}

In this section, we consider a little bit more precisely the conditions satisfied by sub and
supersolutions of the Bellman Equation at time $t=0$ according to Definition~\ref{def:sub.sup.gen}.

In the classical cases where one has a standard initial data $u_0$, these conditions read 
$ \min(\F_*,u-u_0) \leq 0$ for the subsolution and $\max(\F,v-u_0) \geq 0$ for the supersolution, 
and it is known that they just reduce to either $u \leq u_0$ in $\R^N$ if $u$
is a subsolution or $v \geq u_0$ in $\R^N$ if $v$ is a supersolution. 

Here we have an analogous result but which is more complicated, involving the initial
Hamiltonian $\F_{init}$ defined in Section~\ref{IS-BE}.  

\subsection{The general result}

The result is the following.

\begin{proposition}\label{visc-ineq-init}
    Under assumption \HBCL, if $u: \R^N \times [0,\Tf] \to \R$ is an \usc viscosity subsolution of
    the Bellman Equation $\F=0$, then $u(x,0)$  is a subsolution in $\R^N$ of
    $$ (\F_{init})_* \big(x,u(x,0),D_x u(x,0)\big)\leq 0 \quad\hbox{in  }\R^N\;.$$
    Similarly, if $v: \R^N \times [0,\Tf] \to \R$ is a \lsc supersolution of the Bellman Equation,
    then $v(x,0)$ is a supersolution of 
    $$\F_{init}(x,v(x,0),D_x v(x,0) )\geq 0 \quad\hbox{in }\R^N\;.$$
\end{proposition}

\begin{proof}
    We provide the full proof in the supersolution case and we will add additional comments in the
    subsolution one.  Let $\phi : \R^N\to \R$ be a smooth function and let $x$ be a local strict
    minimum point of the function $y\mapsto v(y,0) -\phi(y)$.  In order to use the supersolution
    property of $v$, we consider for $0<\e \ll 1$ the function $(y,t) \mapsto v(y,t) -\phi(y) +
    \e^{-1}t$. 

    By an easy application of Lemma~\ref{lem:cv-pen} in a compact neighborhood of $(x,0)$---with a
    straightforward adaptation to the case of minimas---,this function has a local minimum point at
    $(\xe,\te)$ and we have at the same time $(\xe,\te)\to (x,0)$ and $v(\xe,\te)\to v(x,0)$ as
    $\e\to 0$. The viscosity supersolution inequality reads 
    $$\sup_{(b,c,l)\in\BCL(\xe,\te)}\big\{ \e^{-1}b^t -b^x \cdot
    D_x\phi(\xe) + c v(\xe,\te) - l \big\}\geq 0\; .$$
    We denote by $(b_\e,c_\e,l_\e)$ the $(b,c,l)$ for which the supremum is achieved and which exists since $BCL(\xe,\te)$ is
    compact. By Assumptions
    \HBCL, we may assume that up to extraction, $(b_\e,c_\e,l_\e)\to (\bar b, \bar c,\bar l) \in
    \BCL(x,0)$. Moreover, since $b^t_\e \leq 0$ and the other terms are bounded, the above
    inequality implies that $\e^{-1}b^t_\e$ is also bounded independently of $\e$. In other words,
    $b^t_\e = O(\e)$ and $\bar b= (\bar b^x,0)$. 

    Dropping the negative $\e^{-1}b^t_\e$-term in the supersolution inequality, we obtain 
    $$-b^x_\e \cdot D_x\phi(\xe) + c_\e v(\xe,\te) - l_\e \geq 0\; ,$$ 
    and letting $\e\to0$, we end up with 
    $-\bar b^x \cdot D_x\phi(x) + \bar c v(x,0) - \bar l \geq 0\;.$ 
    since $(\bar b, \bar c,\bar l) \in \BCL(x,0)$, we deduce that
    $$\sup_{((b^x,0),c,l)\in\BCL(x,0)}\big\{-b^x \cdot D_x\phi(x) + c v(x,0) - l \big\}\geq 0\;,$$
    in other words: $\F_{init}\big(x,v(x,0),D_x v(x,0)\big)\geq0$ holds in the viscosity sense.

    \smallskip

    In the subsolution case, the proof is analogous but we consider local strict maximum point of
    the function $y\mapsto u(y,0) -\phi(y)$. Introducing the function $(y,t) \mapsto u(y,t) -\phi(y)
    - \e^{-1}t$ for $0<\e \ll 1$, we have a sequence of local maximas $(\xe,\te)$ such that
    $(\xe,\te)\to (x,0)$ and $u(\xe,\te)\to u(x,0)$ as $\e\to 0$.

    If $\te >0$, the subsolution inequality reads
    \begin{equation*}
        \F_*(\xe,\te, u(\xe,\te), (D_x\phi(\xe),\e^{-1}))\leq 0\;.
    \end{equation*}
    This time, we cannot bound $\e^{-1}b^t$ as we did for the supersolution case, but because of
    \HBCLb-$(i)$, in all $\BCL(x,t)$ for $t\geq 0$, there exists an element with $b^t=-1$.  Since
    the other terms are bounded, this implies that the $\F_*$-term in the above inequality is larger
    than $\e^{-1} +O(1)$ and therefore, for $\e$ small enough, the $\F_*$-inequality above cannot
    hold.

    Hence, necessarily $\te = 0$ and the strict maximum point property for $u-\phi$ implies that
    $\xe=x$. But for the same reason as above, for $\e>0$ small enough the viscosity inequality
    $$ \F_*(x,0,u(x,0),(D_x\phi(x),\e^{-1}))\leq0$$
    cannot hold unless it corresponds to a $(b,c,l)\in\BCL(x,0)$ 
    such that $b^t=0$. Which leads finally to
    $$(\F_{init})_* (x,u(x,0),D_x \phi(x))\leq 0\;,$$
    the inequality we wanted to prove. 
\end{proof}

The above result means that, in order to compute the initial data, one has to solve an equation. A
fact which is already known in the case of unbounded control. 

In the case of classical problems, a typical situation is when for $t>0$, the elements of
$\BCL(x,t)$ are of the form $((b^x,-1),c,l)$ while for $t=0$ we consider a \lsc cost $u_0$ in $\R^N$.
In order to satisfy the upper semi-continuity of $\BCL$ at $t=0$, we need a priori to consider both
elements of the form $((b^x,-1),c,l)$ and $((0,0),1,u_0(x))$. But in that situation, the result
above leads back to the standard initial data conditions
$$ u(x,0) \leq (u_0)^* (x) \quad \hbox{and}\quad v(x,0) \geq u_0(x)\quad \hbox{in  }\R^N\;,$$ 
due to the fact that
$\F_{init}(x,u,p_x)=u-u_0(x)$ and $(\F_{init})_* (x,u,p_x)=u-(u_0)^* (x)$.

\subsection{A relevant example involving unbounded control}\label{sec:eiuc}

As we have seen it above, the general framework we introduce in Section~\ref{Gen-DCP} allows to treat some unbounded
control problems: this is related to the possibility of having $b^t=0$ in the $\BCL$ which is a striking difference with
Chapter~\ref{chap:BasicFram} (we again refer the reader to the beginning of Chapter~\ref{chap:jumps} for some details).

We want to consider here such a problem that we address from the pde point of view by considering the equation
\begin{equation}\label{eq:sre}
\max(u_t + H(x,t,u,D_x u),|D_x u|-1)=0\quad\hbox{in  }\R^N\times (0,\Tf),
\end{equation}
with an ``initial data'' $g$, a bounded, continuous function in $\R^N$ (we are going to make more
precise what we mean by initial data). Here the Hamiltonian $H$ is still given by
$$H(x,t,r,p):= \sup_{\alpha \in A}\,\left\{-b(x,t, \alpha)\cdot p + c(x,t, \alpha)r-l(x,t, \alpha)\right\}\; ,$$
but the functions $b,c,l$ may be discontinuous. Our first aim is to connect this problem with the
above framework and deduce the key assumptions which have to be imposed on $b,c,l$ in order to have
our assumptions being satisfied.

First we have to give the sets $\BCL$ and to do so, we set, for $x\in \R^N$, $t\in (0,\Tf]$
$$ \BCL_1(x,t):=\{ ((b(x,t, \alpha),-1),c(x,t, \alpha),l(x,t, \alpha))\ : \ \alpha \in A\}\; ,$$
and
$$  \BCL_2 (x,t):=\{ ((\beta,0), 0,1)\ : \ \beta \in \overline{B(0,1)}\}\; .$$
Then we introduce
$$\BCL(x,t)=\cob\bigl(\BCL_1 (x,t)\cup\BCL_2 (x,t)\bigr)\; ,$$
where, if $E\subset \R^k$ for some $k$, $\cob (E)$ denotes the closed convex of $E$; computing
$\F(x,t,r,p)=\sup_{(b,c,l)\in\BCL(x,t)}\big\{ -b\cdot p + c r- l \big\}$, we actually find that, for
any $x,t,r,p_x,p_t$
$$\F(x,t,r,(p_x,p_t))=\max(p_t + H(x,t,u,p_x), |p_x|-1)\; .$$
For $t=0$, we have to add the following set
$$  \BCL_0 (x,0):=\{ ((0,0), 1,g(x))\}\;,$$
so that $\BCL(x,0)=\cob\left(\BCL_0 (x,0) \cup \BCL_1 (x,0)\cup\BCL_2 (x,0)\right)$.

We first consider Assumption \HBCLa which is satisfied if the three functions $b(x,t, \alpha)$,
$c(x,t, \alpha)$, $l(x,t, \alpha)$ are bounded on $\R^N\times [0,\Tf]\times A$ and if $ \BCL_1(x,t)$
has compact, convex images and is upper semi-continuous. Next we remark that \HBCLb obviously holds
and we are going to assume in addition that $c(x,t,\alpha)\geq 0$ for all $x,t,\alpha$ (this is not
really an additional assumption since we can reduce to this case by the $\exp(-Kt)$- change).

Since all these assumptions hold, this means that all the results of Section~\ref{Gen-DCP} also
hold. Moreover we have for the initial data $\F_{init}(x,u,p_x):=\max\big\{|p_x|-1, u-g(x) \big\}$
and therefore the computation of the ``real'' initial data comes from the resolution of the
stationary equation 
\begin{equation}\label{eq:sre:id}
    \max( |D_x u|-1, u-g(x))=0\quad\hbox{in  }\R^N.
\end{equation}

\begin{remark}
    Of course, this example remains completely standard as long as we are in the continuous 
    case---typically under the assumptions \HCP. It will be more interesting when treating examples
    in which we have discontinuities in the dynamics, discount factors and costs; or when the term
    ``$|D_x u|-1$'' is replaced by, for instance, ``$|D_x u|-a(x)$'' where $a(\cdot)$ is a
    discontinuous functions satisfying suitable assumptions, in particular $a(x)\geq \eta >0$ in
    $\R^N$.
\end{remark}

As we mentioned it above, unbounded control problems where the cost has a superlinear growth \wrt the dynamic do not enter
into the present framework: we refer the reader to \cite{CRS22,Reis2022} for results on such discontinuous problems with
quadratic growth.

\section{The sub-dynamic programming principle for subsolutions}
\label{sect:dpp.sub}

\index{Dynamic Programming Principle!for subsolutions}
In this section, we provide a sub-dynamic programming principle for subsolutions of Bellman
Equations, but in a more general form than usual, due to the very general framework we use in
Section~\ref{sect:disc.pb} allowing dynamics to have some $b^t=0$. Roughly speaking, we show that if
a \LCR holds in a suitable subdomain $\OO$ of $\R^N\times [0,\Tf]$ and for a suitable equation, then
subsolutions satisfy a sub-dynamic programming principle inside $\OO$.

This formulation is needed in order to get sub-dynamic principles away from the various manifolds on
which the singularities are located, and to deal with situations where the definition of
``subsolution'' may be different from the standard one: even if, to simplify matter, we write below
the equation in a usual form (cf. \eqref{HJ-evol}), the notion of ``subsolution'' can be either an
Ishii subsolution or a stratified subsolution, depending on the context. These specific sub-dynamic
programming principles will play a key role in the proofs of most of our global comparison results,
via Lemma~\ref{lem:comp.fundamental}. 

In order to be more specific, we consider $(x_0,t_0) \in \R^N \times (0,\Tf]$ and the same equation as 
in the previous section set in $Q^{x_0,t_0}_{r,h}$ for some $r>0$ and $0<h<t_0$, namely
\begin{equation}\label{HJ-evol}
\F(x,t,u,Du) = 0\quad \hbox{on  }Q^{x_0,t_0}_{r,h} \;,
\end{equation}
where $\F$ is defined by \eqref{def.F.global}, and we recall that $Du=(D_xu,u_t)$.
We point out that we assume that $\BCL$ and $\F$ are defined in the whole domain $\R^N\times [0,\Tf]$. 

In the sequel, $\mathcal{M}$ is a closed subset of $\overline{Q^{x_0,t_0}_{r,h}}$ such that 
$(x_0,t_0) \notin \mathcal{M}$ and $\OO= Q^{x_0,t_0}_{r,h} \setminus \mathcal{M} \neq \emptyset$.
We denote by $\mathcal{T}^h_\OO(x_0,t_0)$ the set of trajectories starting from $(x_0,t_0)$, such that 
$(X(s),T(s)) \in \OO$ for all $s \in [0,h]$. For simplicity here, 
we assume that the size of the cylinder satisfies $Mh<r$. This is not restrictive at all since when we use the following
sub-dynamic programming principle, we can always apply it in situations where $r$ is fixed and we can choose a 
smaller $h$.

Our result is the
\begin{theorem}\label{thm:sub.Qk.dpp}\emph{--- Extended sub-dynamic programming principle I.}\smsp
    Let $h,r>0$ be such that $Mh<r$. Let $u$ be a subsolution of \eqref{HJ-evol} and 
    let us assume that, for any continuous function $\psi$ such that $\psi \geq u$ on $\overline{Q^{x_0,t_0}_{r,h}}$,
        a \LCR holds in $\OO$ for the equation
        \begin{equation}
            \label{Obst:DPP-sub} \max(\F(x,t,w,Dw), w-\psi) = 0\quad \hbox{in  }\OO\;.
        \end{equation}
    If $\mathcal{T}^h_\OO(x_0,t_0)\neq \emptyset$, then for any $\eta \leq h$
    \begin{equation}\label{ineq:sub.Qk.dpp}
        u(x_0,t_0) \leq 
        \inf_{X \in \mathcal{T}^h_\OO(x_0,t_0)}
        \Big\{\int_0^{\eta}
        l\big(X(s),T(s)\big)\exp(-D(s)) \ds+u\big(X(\eta),T(\eta)\big)\exp(-D(\eta))\Big\}\;.
    \end{equation}
\end{theorem}

\begin{proof} In order to prove \eqref{ineq:sub.Qk.dpp}, the strategy is the following:
we build suitable value functions $\ved$, depending on two
small parameters $\e, \delta$ which are supersolutions of some problems of the type $\max(\F(x,t,v,Dv), v-\psi^\delta)
\geq 0$, for some function $\psi^\delta \geq u$ on $\overline{Q^{x_0,t_0}_{r,h}}$. Then, comparing the supersolutions $\ved$
with the subsolution $u$ and choosing properly the parameters $\e, \delta$ we obtain \eqref{ineq:sub.Qk.dpp} 
after using the dynamic programming principle satisfied by $\ved$. 

The main difficulty is that we have a comparison result which is not valid up to $\mathcal{M}$, only in $\mathcal{O}$.
Therefore we need to make sure that the supersolution enjoys suitable properties not only on $\partial Q^{x_0,t_0}_{r,h}$ 
but also on $\mathcal{M}$.

To do so, we introduce a control problem in $\R^N\times [t_0 - h,t_0]$ with a large penalization 
both in a neighborhood of $\partial Q^{x_0,t_0}_{r,h}$ and outside
$\overline{Q^{x_0,t_0}_{r,h}}$, but also in a neighborhood of $\mathcal{M}$.
Unfortunately, the set valued map $\BCL$ does not necessarily satisfy assumption \HBCLb-$(iii)$ at time $t=t_0-h$, 
which plays the role of the initial time $t=0$ here. We need also to take care of the possibility
that $b^t$ vanishes inside $\overline{Q^{x_0,t_0}_{r,h}}$. For these reasons, we need to enlarge not only the ``restriction'' of  
$\BCL$ to $\R^N \times \{t_0-h\}$ in order to satisfy \HBCLb, but also on the whole domain $\R^N \times [t_0-h,t_0]$. 

For doing so, since $u$ is \usc, it can be approximated a decreasing sequence $(u^\delta)_\delta$ of bounded continuous 
functions and we enlarge $\BCL(x,t)$ for $t \in [t_0-h,t_0]$ by adding elements of the form 
$$((b^x,b^t),c,l)=((0,0),1,u^\delta (x,t)+\delta )\quad \text{for } 0\leq \delta \ll 1\;.$$

On the other hand, we introduce, for $0<\e\ll 1$, the penalization function
$$ \chi_\eps (x,t):= \frac{1}{\eps^4} \Big[\big(2\eps -d((x,t), \mathcal{M}) \big)_+ +(2\e-(r-|x-x_0|))_+ + (2\e-(t-t_0+h))_+\Big],$$
so that  $\chi_\eps (x,t) \geq \eps^{-3}$ if either $d((x,t), \mathcal{M}) \leq \eps$, $d(x,\partial B(x_0,r))\leq \eps$ or
$t-(t_0-h)\leq\e$.

We use this penalization in order to modify the original elements in $\BCL(x,t)$, where $l(x,t)$ is replaced  by  
$l(x,t)+ \chi_\eps(x,t)$. We denote by $\BCL^{\delta,\e}$ this new set-valued map where, at the same time, $\BCL$ is enlarged and modified;
the elements of $\BCL^{\delta,\e}$ are referenced as $(b^{\delta,\e}, c^{\delta,\e}, l^{\delta,\e})$. 
We recall that we can assume that for the original $\BCL$, we have $c\geq 0$ and therefore we also have 
$c^{\delta,\e}\geq 0$ for all $(x,t)$ and  $(b^{\delta,\e}, c^{\delta,\e}, l^{\delta,\e})\in\BCL^{\delta,\e}(x,t)$.

In $\R^N\times [t_0-h,t_0]$, we introduce the value function $\ved$ 
given by
$$
\ved(x,t) = \inf_{\cT^{\delta,\e}(x,t)}\Big\{\int_0^{+\infty}
    l^{\delta,\e} \big(X^{\delta,\e}(s),T^{\delta,\e}(s)\big)\exp(-D^{\delta,\e}(s)) ds \Big\}\; ,$$
where $(X^{\delta,\e},T^{\delta,\e},D^{\delta,\e},L^{\delta,\e})$ are solutions of the differential inclusion associated with $\BCL^{\delta,\e}$, constrained
to stay in $\R^N\times [t_0-h,t_0]$, $\cT^{\delta,\e}(x,t)$ 
standing for the set of such trajectories.

Borrowing arguments from Section~\ref{sect:disc.pb} and computing carefully the new Hamiltonian, 
we see that $\ved$ is a \lsc supersolution of the HJB-equation 
$$\max(\F(x,t,w,Dw),w -(u^\delta+\delta))=0 \quad \hbox{in  }\R^N\times (t_0-h,t_0]\; ,$$
because $l(x,t)+ \chi_\eps(x,t)\geq l(x,t)$ for any $x$ and $t$, and we notice that $u$ is a subsolution of this equation since $u \leq u^\delta +\delta$ in $\R^N\times (t_0-h,t_0]$.
We also remark that, due to the enlargement of $\BCL$, $\ved(x,t)\leq u^\delta(x,t)+\delta$, which is the value obtained by solving
the differential inclusion with $(b,c,l)=((0,0),1,u^\delta(x,t)+\delta)$.
We want to show that $\ved \geq u$ in $\OOb$. In order to do so, 
we have to examine the behavior of $\ved$ in a neighborhood of $\partial \OO$ first, which is provided by the 

\begin{lemma}\label{lem:dpp-sub-boundary}
    For $\e>0$ small enough, $\ved(x,t) \geq  u^\delta(x,t)$ on $\partial\OO$.
\end{lemma}

We postpone the proof of this result and finish the argument. Since $\ved\geq u^\delta\geq u$ on the boundary of $\OO$, 
we have just to look at maximum points of $u-\ved$ in $\OO$ but, in this set, \LCR holds for \eqref{Obst:DPP-sub} 
with $\psi := u^\delta+\delta$. Therefore the comparison is valid and we end up with $\ved\geq u$ everywhere in 
$\overline{\OO}$. 

Ending the proof and getting the sub-dynamic principle is done in three steps as follows.

\noindent\textbf{Step 1 --} at the specific point $(x_0,t_0)$ we have $u(x_0,t_0)\leq \ved (x_0,t_0)$, and using 
the Dynamic programming Principle for $\ved$ at $(x_0,t_0)$ gives
that for any $\eta>0$,
\begin{equation}\label{eq:inf.ved}
u(x_0,t_0)\leq  \inf_{\cT^{\e,\delta}(x_0,t_0)}\Big\{\int_0^{\eta}
    l^{\delta,\e}  \big(X(s),T(s)\big)\exp(-D(s)) ds + \ved \big(X(\eta),T(\eta)\big)\exp(-D(\eta))\Big\}\;.
\end{equation}
we want to get the same inequality, but for trajectories in $\mathcal{T}^h_\OO(x_0,t_0)$. This relies on the following
step.

\noindent\textbf{Step 2 --} \emph{Claim: if $(X,T,D,L)$ is a given trajectory in $\mathcal{T}^h_\OO(x_0,t_0)$ and
if $\eta<h$, then, for $\e>0$ small enough, $(X,T,D,L)$ coincides with a trajectory in $\cT^{\e,\delta}(x_0,t_0)$ on $[0,\eta]$.}

The main argument in order to prove this claim is to notice that
for $\e$ small enough, such trajectories satisfy $\psi_\e(X(s),T(s))=0$ on $[0,\eta]$.

Indeed, let us fix $\eta<h$ and take $\e$ small enough such that $t_0-h+2\e < t_0-\eta$.
Then, for any trajectory $(X,T,D,L)$ in $\mathcal{T}^h_\OO(x_0,t_0)$, $T(s)\in[t_0-\eta,t_0]$ for $s\in[0,\eta]$, so that
$T(s)>t_0-h+2\e$. Similarly, since $Mh<r$ and $|b|\leq M$, we get that $d(X(s);\partial B(x_0,r))>2\e$ for $s\in[0,\eta]$.
Of course, by definition of $\mathcal{T}^h_\OO(x_0,t_0)$, the trajectory does not reach $\mathcal{M}$ hence,
if $\e$ is small enough, 
$d((X(s),T(s));\mathcal{M})>2\e$ for any $s\in[0,\eta]$. In other words, for each fixed trajectory in $\mathcal{T}^h_\OO(x_0,t_0)$, 
if we take $\e$ small enough (depending on the trajectory) we have $\psi_\e(X(s),T(s))=0$ on $[0,\eta]$. 

Therefore, for any trajectory $(X,T,D,L)\in\mathcal{T}^h_\OO(x_0,t_0)$, $l^{\delta,\e}(X(s),T(s))=l(X(s),T(s))$ if $\e>0$ is small enough and $0\leq s \leq \eta <h$.
This means that $(X,T,D,L)$ can be seen as a trajectory associated to the extended $\BCL^{\delta,\e}$, with initial data $(x_0,t_0,0,0)$.
Hence it belongs to $\cT^{\delta,\e}(x,t)$, which proves the claim.

\noindent\textbf{Step 3 --} \emph{Passing to the limit in $\e$ and $\delta$.}

We take a specific trajectory $(X,T,D,L)\in\mathcal{T}^h_\OO(x_0,t_0)$ and take $\e$ small enough so that
we can use it in \eqref{eq:inf.ved}. As we already noticed, $\ved\leq (u^\delta+\delta)$ 
everywhere in $Q^{x_0,t_0}_{r,h}$ due to the enlargement of $\BCL$.
Passing to the limit as $\e\to0$ yields
\begin{equation*}
u(x_0,t_0)\leq \int_0^{\eta}
    l\big(X(s),T(s)\big)\exp(-D(s)) ds + (u^{\delta}+\delta) \big(X(\eta),T(\eta)\big)\exp(-D(\eta))\Big\}\, .
\end{equation*}
Then, we can let $\delta\to 0$ in this inequality, using that $(u^{\delta}+\delta)_\delta$ is a decreasing sequence which converges
to $u$ and
that the trajectory $(X,T,D,L)$ and $\eta$ are fixed.

Therefore $(u^\delta+\delta) \big(X(\eta),T(\eta)\big)\to u \big(X(\eta),T(\eta)\big)$ and we get
$$u(x_0,t_0)\leq  \int_0^{\eta}
    l \big(X(s),T(s)\big)\exp(-D(s)) ds + u\big(X(\eta),T(\eta)\big)\exp(-D(\eta))\;.$$
Taking the infimum over all trajectories in $\mathcal{T}^h_\OO(x_0,t_0)$ yields the conclusion when $\eta<h$. 
The result for $\eta = h$ is obtained by letting $\eta$ tend to $h$, arguing once more
trajectory by trajectory.
\end{proof}

\begin{proof}[Proof of Lemma~\ref{lem:dpp-sub-boundary}] We need to consider three portions of $\partial\OO$: 
$t=t_0-h$, $x \in \partial B(x_0,r)$ and $(x,t) \in \mathcal{M}$. We detail the first estimate which is technically involved, 
then the last two ones are done with similar arguments. In the following, 
we use an optimal trajectory for $v^{\delta,\e}$, denoted by $(X^{\delta,\e},T^{\delta,\e},D^{\delta,\e},L^{\delta,\e})$.

\noindent\textbf{Part A. Initial estimates --} if $t=t_0-h$, we have to consider 

-- the running costs $l \big(X^{\delta,\e} (s),T^{\delta,\e} (s)\big)+ \chi_\eps (X^{\delta,\e} (s),T^{\delta,\e} (s))$, with (perhaps) a non-zero dynamic $b^x$.

-- the running costs $u^\delta(X^{\delta,\e} (s), T^{\delta,\e}(s) )+\delta $ coming from the enlargement with a zero dynamic;

-- and the convex combinations of the two above possibilities, obtained by using a weight $\mu^{\delta,\e}(s)\in[0,1]$.

We first notice that since $t= t_0-h$, we have $T^{\delta,\e}(s)=t_0-h$ for any $s\geq0$ since $b^t\leq 0$ and the
trajectories have the constraint to stay in $\R^N \times [t_0-h,t_0]$. In the following,
we make various estimates (for $\e$ small enough) involving constants $\kappa_0,\kappa_1,\kappa_2,\kappa_3$ 
depending on the datas of the problem and $\delta>0$ but neither on $\e$ nor on $x\in \overline {B(x_0,r)}$.

Next we set
$$E:=\big\{s\in[0,+\infty): l^{\delta,\e} \big(X^{\delta,\e}(s),T^{\delta,\e}(s)\big)=
l^{\delta,\e}\big(X^{\delta,\e}(s),t_0-h\big) \geq \e^{-3/2}\big\}\;,$$
where $l^{\delta,\e}$ is given by the convex combination
$$\begin{aligned}
l^{\delta,\e} \big(X^{\delta,\e}(s),t_0-h\big) &= \mu^{\delta,\e}(s)\Big\{l \big(X^{\delta,\e}(s),t_0-h\big)+ 
\chi_\eps \big(X^{\delta,\e}(s),t_0-h\big)\Big\}\\
& +\big(1-\mu^{\delta,\e}(s)\big) (u^\delta +\delta)\big(X^{\delta,\e} (s), t_0-h\big)
\;.\end{aligned}
$$
By definition of $l^{\delta,\e}$ and in particular because of the $\chi_\eps$-term, we have, for any $s\geq 0$, if $\e$ is small enough
$$ l \big(X^{\delta,\e}(s),t_0-h\big)+ 
\chi_\eps \big(X^{\delta,\e}(s),t_0-h\big)\geq \kappa_0\e^{-3}\; ,$$
while $(1-\mu^{\delta,\e})(u^\delta+\delta)(X^{\delta,\e}(s),t_0-h\big)$ is bounded uniformly with respect to $\e$, $s$ and $x$. Therefore, on $E^c$,
we necessarily have $\mu^{\delta,\e}(s)\leq \kappa_1\e^{3/2}$ for some $\kappa_1>0$.

\

\noindent \emph{Estimates on $E$ --} As we noticed in the proof of Theorem~\ref{thm:sub.Qk.dpp}, $\ved\leq u^\delta+\delta$. In particular,
$$\begin{aligned}
(u^\delta +\delta)(x,0) &\geq \ved(x,0) \\
& \geq \int_0^{+\infty}
    l^{\delta,\e}  \big(X^{\delta,\e}(s),T^{\delta,\e}(s)\big)\exp(-D^{\delta,\e}(s)) ds\\
& \geq  \int_E
    l^{\delta,\e}  \big(X^{\delta,\e}(s),T^{\delta,\e}(s)\big)\exp(-D^{\delta,\e}(s)) ds \\
    & +  \int_{E^c}
    l^{\delta,\e}  \big(X^{\delta,\e}(s),T^{\delta,\e}(s)\big)\exp(-D^{\delta,\e}(s)) ds
\end{aligned}$$
By definition of $E$, the first integral is estimated by
$$\int_E
    l^{\delta,\e}  \big(X^{\delta,\e}(s),T^{\delta,\e}(s)\big)\exp(-D^{\delta,\e}(s)) ds
\geq 
\int_E
    \e^{-3/2} \exp(-D^{\delta,\e}(s)) ds \; ,$$
while, using the boundedness of $l$ and $(u^\delta +\delta)$ there exists $C>0$ such that
$$\int_{E^c}
    l^{\delta,\e}  \big(X^{\delta,\e}(s),T^{\delta,\e}(s)\big)\exp(-D^{\delta,\e}(s)) ds \geq 
 - C \int_{E^c} \exp(-D^{\delta,\e}(s)) ds\; .$$
 To get an estimate on the Lebesgue measure of $E$, we need an upper estimate of $\int_{E^c} \exp(-D^{\delta,\e}(s)) ds$. 
Notice that on $E^c$, because of the estimate on $\mu^{\delta,\e}(s)$ we have
 $$
\dot D^{\delta,\e}(s)=c^{\delta,\e} \big(X^{\delta,\e}(s),T^{\delta,\e}(s)\big) = 
\mu^{\delta,\e}(s) c\big(X^{\delta,\e}(s),T^{\delta,\e}(s)\big)+\big(1-\mu^{\delta,\e}(s)\big) = 1 + O(\e^{3/2})\; ,
$$
where the $|O(\e^{3/2})|\leq M\kappa_1\e^{3/2}$ is independent of $x$.
Hence, since $\dot D^{\delta,\e}(s)\geq 0$ for any $s\geq 0$,
\begin{align}
\int_{E^c} \exp(-D^{\delta,\e}(s)) ds &=
\int_{E^c} \frac{\dot D^{\delta,\e}(s)}{(1 + O(\e^{3/2}))} \exp(-D^{\delta,\e}(s)) ds\\
&\leq  (1 + O(\e^{3/2}))^{-1} \int_0^{+\infty} \dot D^{\delta,\e}(s) \exp(-D^{\delta,\e}(s)) ds\\
&\leq  (1 + O(\e^{3/2}))^{-1}\; .
\end{align}

Gathering all the above informations, we finally conclude that
$$ \int_E
    \e^{-3/2} \exp(-D^{\delta,\e}(s)) ds \leq \kappa_2\; ,$$
for some constant $\kappa_2$ which is independent of $\e$ and $x$.

We introduce now a parameter $S>0$ and denote by $E_S:=E\cap[0,S]$.
Since $0 \leq \dot D^{\delta,\e}(s)\leq M$ for any $s\geq 0$, we have
$$
 \exp(-MS)|E_S| \leq \int_{E_S}
     \exp(-D^{\delta,\e}(s)) ds \leq \int_E
     \exp(-D^{\delta,\e}(s)) ds \leq \kappa_2\e^{3/2}\; ,
$$
where $|E_S|$ denotes the Lebesgue measure of $E_S$. We choose $S=\Se$ such that $\exp(M\Se)=\e^{-1/6}$ which yields
$$|E_\Se| \leq \kappa_2 \e^{3/2}\exp(M\Se)= \kappa_2\e^{4/3}\; .$$
We remark that $\Se$ behaves like $\ln(\eps^{-1/6})$, uniformly in $x$. The reason why 
we choose $\Se$ in order to get a power $4/3>1$ in $|E_\Se|$ will become clear in the lateral estimates. For Part A, 
any power in $(0,3/2)$ is convenient. \\
 
\noindent \emph{Consequences on $\ved$ --} We first apply the Dynamic Programming Principle for $\ved$ which gives
\begin{align}
\ved (x,t_0-h)= &\int_0^{\Se}
    l^{\delta,\e} \big(X^{\delta,\e}(s),t_0-h\big)\exp(-D^{\delta,\e}(s)) ds \\ &
    + \ved(X^{\delta,\e}(\Se),t_0-h)\exp(-D^{\delta,\e}(\Se))\; .
\end{align}
Now we have to examine each term carefully. 
We first come back to the equation of $D^{\delta,\e}$: 
we have seen above that $|\dot D^{\delta,\e}(s)-1|\leq M\kappa_1\e^{3/2}$ on $E^c$, while $|E_\Se|\leq\kappa_2\e^{4/3}$. We deduce that, for $s \in [0,\Se]$
\begin{equation}\label{eq:D}
 |D^{\delta,\e}(s)-s|\leq M(\kappa_1\e^{3/2}\Se+\kappa_2\e^{4/3})\leq\kappa_3\e^{4/3}
\end{equation}
for some $\kappa_3>0$.
In particular, since $\Se\to +\infty$ as $\e\to 0$, $\exp(-D^{\delta,\e}(\Se)) \to 0$ as $\e\to 0$ and 
$$\liminf_{\e\to 0}\left(\ved(X^{\delta,\e}(S),t_0-h)\exp(-D^{\delta,\e}(\Se))\right)\geq 0 \;,$$
 uniformly w.r.t. $x$ since $\ved$ is bounded from below.

On an other hand, for the $X^{\delta,\e}$-equation, we also have, on $E^c$ (in fact only the $b^x$ part is useful here)
$$b^{\delta,\e}(X^{\delta,\e}(s),t_0-h)=\mu^{\delta,\e}(s)b(X^{\delta,\e}(s),t_0-h)+(1-\mu^{\delta,\e}(s))(0,0)= O(\e^{3/2})\; , $$
more precisely the bound takes the form $M\kappa_2\e^{3/2}$.
Using the decomposition with $E_\Se$ and its complementary $E_\Se^c=E^c\cap[0,\Se]$ as in \eqref{eq:D}, it follows that
$$\begin{aligned}
    \int_0^\Se|b^{\delta,\e}(\tau)| d \tau & =\int_0^\Se|b^{\delta,\e}(\tau)| \ind{E_\Se}(s)d \tau+\int_0^\Se|b^{\delta,\e}(\tau)|\ind{E_\Se^c}(s)\,d \tau\\
     &\leq M(\kappa_2\e^{4/3}+\kappa_1\e^{3/2}\Se)\leq\kappa_3\e^{4/3}\;.
\end{aligned}$$
We deduce that if $s \in [0, \Se]$, $X^{\delta,\e} (s)-x = O(\e^{4/3})$ and since $u^\delta$ is continuous, 
$$ (u^\delta +\delta)(X^{\delta,\e} (s), t_0-h)=(u^\delta +\delta)(x, t_0-h) + o_\e(1)\geq (u^\delta +\delta/2)(x, t_0-h)\;.$$
For a similar reason, on $E^c_\Se$ we can absorb the $o_\e(1)$-term by a $\delta/2$ for $\e$ small enough
$$l^{\delta,\e} \big(X^{\delta,\e}(s),t_0-h\big) \geq (u^\delta +\delta/2)\big(x, t_0-h\big)\;.$$

Gathering all these informations, using \eqref{eq:D} and that $(l+\psi_\e)\geq0$ on $E_\Se$ we get
$$\begin{aligned}
I_\e &:= \int_0^{\Se}
    l^{\delta,\e}  \big(X^{\delta,\e}(s),t_0-h\big) \exp(-D^{\delta,\e}(s)) ds \\
  &\geq \int_{E^c_\Se}
    \Big((u^\delta +\delta/2)\big(x, t_0-h\big) \Big) \exp\big(-s+O(\e^{4/3})\big) ds\;.
\end{aligned}$$
Then, since $\Se$ behaves like $\ln(\e^{-1/6})$ and $|E_\Se|\leq\kappa_2\e^{4/3}$, we get 
$$\begin{aligned}
     I_\e & \geq (u^\delta +\delta/2) \big(x, t_0-h\big) \int_{E^c_\Se} \exp(-s) ds + o_\e(1) \\
 & \geq (u^\delta +\delta/2) \big(x, t_0-h\big)+ o_\e(1)\; .
\end{aligned}
$$
Hence $\ved (x,t_0-h) \geq (u^\delta +\delta/2) \big(x, t_0-h\big)+ o_\e(1)$ where the ``$o_\e(1)$'' is independent of $x$
and for $\e$ small enough, we have $\ved (x,t_0-h) \geq u^\delta \big(x, t_0-h\big)$ on $\overline{B(x_0,r)}$.\\

\noindent \textbf{Part B. Lateral estimates --} Essentially, the proof is the same as for the initial estimates:
the only difference is that the trajectory may exit the region where $\chi_\eps$ is large. But,
if $d((x,t), \mathcal{M}) \leq \eps$ or if $d(x,\partial B(x_0,r))\leq \eps$, the running cost satisfies again the estimate
$l \big(X^{\delta,\e} (s),T^{\delta,\e} (s)\big)+ \chi_\eps (X^{\delta,\e} (s),T^{\delta,\e} (s))\geq
\kappa_0\e^{-3}\geq0$.

We consider the case when $(x,t) \in \mathcal{M}$, the proof being the same if $(x,t)\in\partial B(x_0,r)$. 
Since the dynamic $b$ is bounded by $M$, a trajectory $(X,T)$ starting at $(x,t)$ satisfies
$d((X(s),T(s)),\mathcal{M})\leq Ms$ and therefore,  it stays in an $\e$-neighborhood of $\mathcal{M}$ for $s<\eps/M$.

For an optimal trajectory, we repeat the same proof as in Part A, but on $E\cap [0,\tau_\e\wedge\Se]$, where $\tau_\e$ is the first time for which 
$d((X^{\delta,\e}(s),T^{\delta,\e}(s)),\mathcal{M})= \eps $ and $a\wedge b=\min(a,b)$.

If we set as above
$$
    E:=\big\{s\in[0,\infty): l^{\delta,\e} \big(X^{\delta,\e}(s),T^{\delta,\e}(s)\big) \geq \e^{-3/2}\big\}\;,
$$
then the Lebesgue measure of $E\cap [0,\tau_\e\wedge\Se]$ is less than $\kappa_3\e^{4/3}$ for some $\kappa_3>0$,
while on $E^c\cap[0,\tau_\e\wedge S_\e]$ we have $\mu^{\delta,\e}(s)\leq \kappa_4 \e^{3/2}$ for some $\kappa_4>0$.
As in Part A, using the decomposition on $E\cap [0,\tau_\e\wedge\Se]$ and its complementary we deduce that 
$$\int_0^{\tau_\e\wedge S_\e}|b^{\delta,\e}(s)|\,ds\leq\,M\big\{\kappa_3\e^{4/3}+\kappa_4\e^{3/2}(\tau_\e\wedge\Se)\big\}\;,$$
while by definition the distance between $(x,t)$ and $(X^{\delta,\e}(\tau_\e),T^{\delta,\e}(\tau_\e))$ is larger than $\e$ 
(if $\tau_\e$ is finite, of course).

We claim that for $\e$ small enough, $\tau_\e\wedge\Se=\Se$. 
Indeed, assume on the contrary that for some subsequence $\e_n\to0$,
$\tau_{\e_n}<S_{\e_n}$. From the previous estimate it follows that 
$$\e_n\leq |(X^{\delta,\e}(\tau_\e),T^{\delta,\e}(\tau_\e))-(x,t)|\leq M\big\{\kappa_3\e_n^{4/3}+\kappa_4\e_n^{3/2}\tau_{\e_n}\big\}\;.$$
The fact that the power in the first term is greater than 1 implies that $\tau_{\e_n}$ goes to infinity, at least like $\e_n^{-1/2}$.
But since by construction $S_{\e_n}$ behaves like $\ln(\e_n^{-1/6})$, we reach a contradiction.

We deduce that necessarily $\tau_\e>\Se$ as $\e\to 0$, and that on $[0,\Se]$, the trajectory 
remains ``trapped'' in an $\e$-neighborhood of $\mathcal{M}$. 
We end the proof exactly as in Part A, sending $\e\to0$.

The proof if $x\in\partial B(x_0,r)$ being the same, in conclusion 
we have shown that $v^{\delta,\e} \geq u^\delta$ on $\partial \OO$ for $\e$ small enough.
\end{proof}   

\

In the case when $b^t$ is not allowed to vanish, obtaining the sub-dynamic principle 
is a bit easier since we do not need to consider an obstacle-type problem like \eqref{Obst:DPP-sub}.
\begin{theorem}\label{thm:sub.Qk.dpp.bt}\emph{--- Extended sub-dynamic programming principle II.}\smsp
    Let $h,r>0$ be such that $Mh<r$ and assume that, for any 
$(x,t) \in \overline{Q^{x_0,t_0}_{r,h}}$ and any $(b,c,l) \in \BCL(x,t)$, $b^t=-1$. If $u$ is a subsolution of \eqref{HJ-evol}, if
    $\mathcal{T}^h_\OO(x_0,t_0)\neq \emptyset$ and if a \LCR holds in $\OO$ for the equation $\F=0$, then for any $\eta \leq h$
    \begin{equation}\label{ineq:sub.Qk.dpp2}
        u(x_0,t_0) \leq 
        \inf_{X \in \mathcal{T}^h_\OO(x_0,t_0)}
        \Big\{\int_0^{\eta}
        l\big(X(s),T(s)\big)\exp(-D(s)) \ds+u\big(X(\eta),T(\eta)\big)\exp(-D(\eta))\Big\}\;.
    \end{equation}
\end{theorem}

 \begin{proof}The difference between the two cases comes from the fact that, under the assumption of 
 Theorem~\ref{thm:sub.Qk.dpp.bt}, we could have $T(h)>t_0-h$ in \eqref{eq:inf.ved} (Step~1) for a trajectory starting from
 $(x_0,t_0)$ since $b^t$ was allowed to be different from $-1$: this is why the strategy of the proof of this theorem
 uses $\eta<h$ and, for handling this situation, we need to have $\ved (x,t) \leq u^\delta(x,t) +\delta$ in 
 the whole domain to conclude after using the Dynamic Programming Principle for $\ved$ (cf. Step~3).
 
Here on the contrary we are sure that $T(h)=t_0-h$ for any such
trajectory and we are going the Dynamic Programming Principle for $\ved$ up to time $t_0-h$, i.e. with $s=h$.

 For this reason, we are going to prove \eqref{ineq:sub.Qk.dpp2} for $\eta=h$, the inequality for $\eta<h$ being obtained by applying the result with $h$ replaced by $\eta$.

For all these reasons the proof is similar to that of Theorem~\ref{thm:sub.Qk.dpp} but there are substantial simplifications. 

\noindent\textbf{(a)} We enlarge $\BCL$ in the same way BUT ONLY at time  $t=t_0-h$.
The consequence is that $\ved$ is a supersolution for the HJB-equation $\F=0$ and not of \eqref{Obst:DPP-sub}, since we have no enlargement for $t\in (t_0-h,t_0)$. Hence we just have to deal with the comparison results for the $\F$-equation, we do not need to assume some obstacle-type comparison property.

\noindent\textbf{(b)} The penalization function we use here does not require a specific penalization for the initial time and we just write it as
$$ \chi_\eps (x,t):= \frac{1}{\eps^4} \Big[\big(2\eps -d((x,t), \mathcal{M}) \big)_+ +(2\e-(r-|x-x_0|))_+\Big]\;.$$
The initial inequality $\ved(x,t_0-h)\geq (u^\delta+\delta)(x,t_0-h)$ for any $x\in B(x_0,r)$ follows from the following argument:
since $b^t=-1$ in $\BCL$, the only possibility for a constrained trajectory $(X^{\delta,\e},T^{\delta,\e},D^{\delta,\e},L^{\delta,\e})\in\cT^{\delta,\e}(x,t_0-h)$ to remain in $\R^N \times [t_0-h,t_0]$ is to solve the differential inclusion
by using the elements $((0,0),1,(u^\delta+\delta)(x,t_0-h))$ of $\BCL^{\delta,\e}$. 
This implies directly that $\ved(x,t_0-h)\geq (u^\delta+\delta)(x,t_0-h)$.

\noindent\textbf{(c)} With these simplifications, the proof remains the same as in the general case $b^t\in[-1,0]$: we first get
 that $\ved\geq u$ on $t=t_0-h$, for $x\in\partial B(x_0,r)$ and for $(x,t) \in \mathcal{M}$. Using that we have a \LCR in 
 $\OO$ implies that $\ved\geq u$ on $\overline{\OO}$. Then we proceed as above using the dynamic programming 
 principle for $\ved$. For $\eta\leq h$\footnote{Here we do not have to treat separately the cases when $\eta< h$ and $\eta
 = h$ since we have dropped the penalization term in a neighborhood of $t=t_0-h$ and we know that $\ved(x,t_0-h)\geq (u^\delta+\delta)(x,t_0-h)$.}, taking $\e>0$ small enough
allows to restrict this dynamic principle to the trajectories in $\mathcal{T}^h_\OO(x_0,t_0)$, which avoid $\mathcal{M}$. 
Sending $\e\to0$ and $\delta\to0$ is done ``trajectory by trajectory''.
\end{proof}

\section{Local comparison for discontinuous HJB Equations}
\label{sect:Local.Comparison}

The aim of this section is to provide an argument which is a keystone in several comparison results we
give for HJB Equations with discontinuities, and in particular for stratified problems.

To do so, we consider a $C^1$-manifold $\mathcal{M}\subset\R^N\times (0,\Tf)$ (which will be in the sequel a set of discontinuity for the HJB Equation) and
for any $(x,t) \in \mathcal{M}$, we denote by $T_{(x,t)} \mathcal{M}$, the tangent space of $\mathcal{M}$ at $(x,t)$. Then we define the \emph{tangential Hamiltonian} associated 
with $\mathcal{M}$ by setting
\begin{equation}\label{def:tangent.Hamiltonian}
	\F^{\mathcal{M}}(x,t,u,p):=\sup_{(b,c,l)\in\BCL_T(x,t)} \big\{-b\cdot p + cu-l\big\}\;,
\end{equation}
where $\BCL_T(x,t):=\big\{(b,c,l)\in\BCL(x,t):b\in T_{(x,t)} \mathcal{M}\big\}$. This tangential Hamiltonian is defined for any
$(x,t)\in \mathcal{M}\times[0,\Tf]$, $u\in\R$ and $p\in T_{(x,t)} \mathcal{M}$. But by a slight abuse of notation, we also write
$\F^{\mathcal{M}} (x,t,u,p)$ when $p\in\R^{N+1}$, meaning that only the projection of $p$ onto $T_{(x,t)} \mathcal{M}$ is used for the computation.
We also recall that $Du=(D_xu,u_t)$.

Our main argument comes from the\index{Comparison result!a general lemma for}\index{Dynamic Programming Principle!in the Magical Lemma}
\index{Magical Lemma}
\begin{lemma}\label{lem:comp.fundamental}\emph{--- The ``Magical Lemma''.}\smsp
Assume that \HBCL holds and fix $(x,t)\in \mathcal{M}$, $0<t-h<t\leq \Tf$. Assume that $v:\cylb\to\R$ is a \lsc
    supersolution of $\F(x,t,v,Dv)=0$ in $\cyl$ and $u:\cylb\to\R$ has the following properties:
\begin{enumerate}
\item[$(i)$] $u\in C^0(\cylb)\cap C^1(\mathcal{M})$,
\item[$(ii)$] $\F^{\mathcal{M}}(y,s,u,Du)<0\text{ on }\mathcal{M}$,
\item[$(iii)$] $u$ satisfies a ``strict'' subdynamic principle in $\cyl[\mathcal{M}^c]=(B(x,r)\times(t-h,t])\setminus \mathcal{M}$, i.e.
there exists $\eta >0$, such that, for any $(\xb,\tb) \in \cyl[\mathcal{M}^c]$, for any solution $(X,T,D,L)$ of the differential inclusion such that $X(0)=\xb$, $T(0)=\tb$ and
$(X(s),T(s))\in \cyl[\mathcal{M}^c]$ for $0< s \leq \bar \tau$, we have, for any $0<\tau\leq \bar \tau$
\begin{equation}\label{strict-sdpp}
u(\xb,\tb)\leq  \int_0^\tau (l(X(s),T(s))-\eta)\exp(-D(s))\ds + u(X(\tau),T(\tau))\exp(-D(\tau)).
\end{equation}
	\end{enumerate}
If $\displaystyle \max_{\overline{\cyl}}(u-v) >0$, then, for any $(y,s)\in \overline{\cyl} \setminus \partial_p \cyl $, 
$$(u-v )(y,s)<m:=\max\limits_{\partial_p \cyl}(u-v)\, .$$
\end{lemma}

\begin{proof} Using \HBCLb, we can assume without loss of generality that $c\geq 0$ for all $(b,c,l)\in\BCL(y,s)$ and for all $(y,s) \in  \overline{\cyl}$.

We assume by contradiction that $(u-v)$ reaches its maximum on $\cylb$ at a point $(\xb,\tb)\in \cyl$. If $(\xb,\tb)\in \cyl \setminus \mathcal{M}$, we easily reach a
contradiction: by Lemma~ \ref{lem:super.dpp}, $v$ satisfies \eqref{ineq:super.dpp} and for sufficiently small $\tau$, all the trajectories $(X,T,D,L)$ are such that $(X(s),T(s))\in \cyl[\mathcal{M}^c]$. We consider an optimal trajectory for $v$ at $(\xb,\tb)$, $(X,T,D,L)$ and we gather the information given by \eqref{ineq:super.dpp} and \eqref{strict-sdpp} for some time $\tau$ small enough: substracting
these inequalities, we get
\begin{equation}\label{eq:dpp.comp.local}
	u(\xb,\tb)-v(\xb,\tb)\leq -\eta \tau + (u(X(\tau),T(\tau))-v(X(\tau),T(\tau)))\exp(-D(\tau))\;.
	\end{equation}
But $(\xb,\tb)$ is a maximum point of $u-v$ in $\cylb$ and therefore we have at the same time $u(\xb,\tb)-v(\xb,\tb)>0$ and 
$u(\xb,\tb)-v(\xb,\tb) \geq u(X(\tau),T(\tau))-v(X(\tau),T(\tau))$; hence, since $\exp(-D(\tau))\geq 0$
$$
u(\xb,\tb)-v(\xb,\tb)\leq -\eta \tau + (u(\xb,\tb)-v(\xb,\tb))\exp(-D(\tau))\;,$$
which is a contradiction since $\exp(-D(\tau))\leq 1$.

If $(u-v)$ reaches its maximum on $\cylb$ at a point $(\xb,\tb)\in \cyl \cap  \mathcal{M}$, we face two cases

\noindent\textbf{A. --} In \eqref{ineq:super.dpp} for $(\xb,\tb)$, there exists a trajectory $(X,T,D,L)$ and $\tau>0$ such that $X(0)=\xb$, $T(0)=\tb$ and
\begin{equation}\label{sup-dpp-v-tau}
v(\xb,\tb)\geq \int_0^{\tau}
        l\big(X(s),T(s)\big)\exp(-D(s)) \ds+v\big(X(\tau),T(\tau)\big)\exp(-D(\tau))\; ,
\end{equation}
AND $(X(s),T(s))\in \cyl \setminus \mathcal{M}$ for $s\in(0,\tau]$. In this case we argue essentially as above: we use as a starting point $(x_\e,t_\e):=(X(\e),T(\e)) \in \cyl[\mathcal{M}^c] $ for $0<\e \ll 1$ and we use \eqref{strict-sdpp} for the specific trajectory $(X,T,D,L)$ but on the time interval $[\e,\tau]$
$$u(x_\e,t_\e) \leq  \int_\e^\tau (l(X(s),T(s))-\eta)\exp(-D(s))\ds + u(X(\tau),T(\tau))\exp(-D(\tau))\;.$$
But in this inequality, we can send $\e$ to $0$, using the continuity of $u$ and finally get, combining it with the above inequality for $v$
to obtain \eqref{eq:dpp.comp.local} and a contradiction.

\noindent\textbf{B. --} If Case~A cannot hold, this means that, for any $\tau$ and for any trajectory $(X,T,D,L)$ such that \eqref{sup-dpp-v-tau} holds, then there exists a sequence $t_n\searrow 0$ such that 
$X(t_n)\in \mathcal{M}$ for any $n\in\N$. \index{Dynamic Programming Principle!for supersolutions}We first use the dynamic programming inequality for $v$
between $s=0$ and $s=t_n$, which yields
$$v(\xb,\tb)\geq \int_{0}^{t_n} l(X(s),T(s))\exp(-D(s))\ds + v(X(t_n),T(t_n))\exp(-D(t_n))\;.$$
Since $u-v$ reaches a maximum at $(\xb,\tb)$ and since this maximum is positive, we can replace $v$ by $u$ in this inequality which leads to
$$\frac{u(\xb,\tb)-u(X(t_n),T(t_n))\exp(-D(t_n))}{t_n}\geq \frac{1}{t_n}\int_{0}^{t_n} l(X(s),T(s))\exp(-D(s))\ds\;.$$
Now, since $u$ is $C^1$-smooth on $\mathcal{M}\times (t-h,t)$, we have (recall that $Du=(D_x u,u_t)$ and that here we use only derivatives which are in the tangent space of $\mathcal{M}$)
\begin{align}
u(X(t_n),T(t_n))=& u(\xb,\tb)+ D u(\xb,\tb)(X(t_n)-\xb,T(t_n)-\tb)+  o(|X(t_n)-\xb|+|T(t_n)-\tb|)\\
 = & u(\xb,\tb)+ D u(\xb,\tb)(X(t_n)-\xb,T(t_n)-\tb)+ o(t_n)\; ,
\end{align}
and writing
$$ (X(t_n)-\xb,T(t_n)-\tb) = \int_{0}^{t_n} b(s)ds\; , \; \exp(-D(t_n)) = \int_{0}^{t_n} - c(s)\exp(-D(s))ds $$
we obtain
$$ 
\frac{1}{t_n}\int_{0}^{t_n}\left\{ -b(s)\cdot Du(\xb,\tb) + c(s)u(\xb,\tb) - l(X(s),T(s))\right\}\exp(-D(s)) ds \geq 0 \; .
$$
And since $\exp(-D(s))=1+O(t_n)$, we can write this inequality as
$$-b_n \cdot Du(\xb,\tb) +  c_n u(\xb,\tb) -l_n
  \geq o_n(1) \; ,
$$
where
$$ b_n=\left(\frac{1}{t_n}\int_{0}^{t_n} b(s)ds\right)\; ,\; c_n=\left(\frac{1}{t_n}\int_{0}^{t_n} c(s)ds\right)\; ,\; l_n\left(\frac{1}{t_n}\int_{0}^{t_n} l(X(s),T(s))ds\right).$$
But the $b_n, c_n,l_n$ are uniformly bounded and therefore we can assume that $b_n\to \bar b, c_n\to \bar c,l_n\to \bar l$. Using the convexity and upper semi-continuity of $\BCL$, we have $(\bar b, \bar c, \bar l)\in \BCL(\xb,\tb)$ and by the definition of $b_n$, we also have $\bar b \in T_{(\xb,\tb)} \mathcal{M}$. Finally, passing to the limit in the above inequality yields
$$ -\bar b\cdot Du(\xb,\tb) + \bar c u(\xb,\tb) - \bar l \geq 0\; .$$
But, thanks to the definition of $ \F^{\mathcal{M}}$ and the properties of $u$, we have the inequalities
$$ 0\leq -\bar b\cdot Du(\xb,\tb) + \bar c u(\xb,\tb) - \bar l \leq  \F^{\mathcal{M}}(\xb,\tb,u(\xb,\tb),Du(\xb,\tb))<0\; ,$$
which is the desired contradiction.
\end{proof}

\begin{remark}\label{rem:comp.fundamental} 
There are possible variants for this lemma. In particular, in Part~\ref{part:codim1}, we use one of them where the sub and supersolution properties 
for $u$ and $v$ are defined in a slightly different way, namely with taking a more restrictive set of control on $\mathcal{M}$. Of course, in that case,
$\F^{\mathcal{M}}$ is replaced by an Hamiltonian which defined in a different way. The proof is still valid if the Dynamic Programming argument of 
\textbf{B.} leads to the right inequality.
\end{remark}

\section{The ``good framework for HJ Equations with discontinuities''}
\label{sec:GFHJD}
\index{Good framework for HJE}
\label{page:GF.HJE}

The study of Hamilton-Jacobi with discontinuities or the associated control problems in the convex
case leads to various situations, many of which we consider in Parts~\ref{part:codim1},~\ref{part:NA},
~\ref{stratRN} or~\ref{S-BC}. These situations may appear to be quite different, but still we can
identify some common structure on the equations and the discontinuities of the Hamiltonians which
seems quite ``natural'' to get most of the results. Of course, what we are going to describe as the
``good framework for HJ-Equations with discontinuities'' does not perfectly fit all situations and
some adaptations have to be made in each case. But the definition below provides a good idea of the
key assumptions which are required to treat those problems.

\subsection{General definition at the pde level}

\index{Good framework for HJE!at the pde level}
\label{page:GF.HJE:pde}
\begin{definition}\label{GF-HJD}\emph{--- The good framework for HJ-Equations.}\smsp
    We say that we are in the ``good framework for HJ-Equations with discontinuities'' for the
    equation
    \begin{equation}\label{eq:GF-HJD}
    \G(X,u,Du)=0\quad \hbox{in  }\OO\subset \R^{N}
    \end{equation}
    if \LOCa, \LOCb hold and if there exists an \TFS $\M=(\Man{k})_{k=0..N}$ of $\R^{N}$
    such that, for any $k=0,..,N$ 
    \begin{enumerate}
        \item[$(i)$] if $\bar X \in \Man{k}\cap \OO$, there is a ball $B(\bar X,r)\subset \OO$ for
            some $r>0$ and a $C^{1,1}$-diffeomorphism $\Psi:B(\bar X,r) \to \R^N$ such that
            $\Psi(\bar X)=\bar X$,
            $$\Psi(B(\bar X,r)\cap\Man{k})=\big(\bar X+\R^k\times \{0_{\R^{N-k}}\}\big)\cap
            \Psi(B(\bar X,r))\, ,$$
            and, for any $l=(k+1)..N$ and $\bar Y\in \Psi(\Man{l}\cap B(x,r))$
            $$ \big(\bar Y+\R^k\times \{0_{\R^{N-k}}\}\big)\cap
            \Psi(B(\bar X,r))\subset \Psi(\Man{l}\cap B(x,r))\quad \hbox{.}$$
    \item[$(ii)$] Denoting $\Psi(X)=\bar X+(Y,Z)$ with $Y\in \R^k$, $Z\in \R^{N-k}$ and 
        $$ \tilde \G ((Y,Z),r,(p_Y,p_Z))= \G(\Psi^{-1}\left(\bar X+(Y,Z)\right),r,
            [(\Psi^{-1})']^T \left(\bar X+(Y,Z)\right)(p_Y,p_Z))\;,$$ 
        where $[(\Psi^{-1})']^T$ denotes the transpose matrix of $(\Psi^{-1})'$, then \TC, \NCe,\Mong
        hold for $\tilde \G$ on $\Psi(B(\bar X,r)\cap\Man{k})$.
    \end{enumerate}
    In this case, we will say that $\M$ is associated to Equation~\eqref{eq:GF-HJD}.
\end{definition}

As we already mentioned it in Section~\ref{sect:sup.reg}, the difficulty when stating such
definition is that it is supposed to cover very different situations for which the sense of $\G=0$
may vary and may also involve several Hamiltonians.  In these various situations, we use the
following convention

\noindent {\em \TC, \Mong have to be satisfied - {\em up to some change of variables} -  by ANY
Hamiltonians which are involved in the sub and supersolutions inequalities while \NCe has to be
satisfied by the Hamiltonians which are involved in the subsolutions inequalities related to local
maximum points in $\OO$---or $\Psi(B(\bar X,r))$---but not by the Hamiltonians related to local
maximum points on the $\Man{k}$ for $k<N$.}

But before coming back to this point, let us explain the key ideas beyond this
``good framework for HJ-Equations with discontinuities''.

The very first idea is that the discontinuities of $\G$ form an \TFS. Since we
always argue locally (using \LOCa, \LOCb, for comparison results), we can use
Definition~\ref{def:TFS}---with perhaps a smaller $r$---to reduce to the case when the
$k$-dimensional discontinuity on $\G$, $\Man{k}$, can be flatten, here replaced by $\bar X+\R^k\times
\{0_{\R^{N-k}}\}$). This is the first important reduction. We immediately point out that, in
Definition~\ref{GF-HJD}, the diffeomorphism $\Psi$ is assumed to be $C^{1,1}$ which is needed in
general to get \TC but, in  coercive cases, i.e. when $G$ is coercive in $p$,
$C^{1}$-diffeomorphisms may be enough.

Once this change is done, we are in the framework of Section~\ref{sect:sup.reg} and using a
combination \TC, \NCe,\Mong allows us to regularize subsolutions in order to be able to apply
Lemma~\ref{lem:comp.fundamental}. The triptych ``\emph{Tangential continuity $+$ normal
controllability $+$ some suitable monotonicity}'' seems to us the basis of most of our results, and
not only the comparison ones.

The two extreme cases have also to be commented: if $k=N$, then there is no normal directions, \TC
has to be satisfied by all coordinates, $\G$ is continuous in a neighborhood of $\bar X$, no change
$\Psi$ is really needed and, through \TC, we just recover the classical assumption  for the
uniqueness of viscosity solutions for a standard HJ-Equations without discontinuity. If $k=0$, $\bar
X$ is an isolated point, we have no  ``tangent coordinates'' and \TC is void but \NCe implies that
$\G$ is coercive in $p$ in a neighborhood of $\bar X$.

\subsection{The stratified case, ``good assumptions'' on the control problem}

Now let us come back on the sense of the equation $\G=0$ and the way the above convention has to be
applied. Anticipating Part~\ref{stratRN} on the full stratified case, we have an HJ Equation of the
type
$$
\F(x,t,U,DU)=0\quad\text{in}\quad\R^N\times[0,\Tf]\;,
$$
where $DU=(D_x U, D_t U)$ and 
$$
\F(x,t,r,p):=\sup_{(b,c,l)\in\BCL(x,t)}\big\{ -b\cdot p + c r- l \big\}\;.
$$
Assuming that \HBCL holds, what does it mean to be in the ``good framework for HJ-Equations with
discontinuities'' here? 

In the case of stratified problems, roughly speaking, the sense of the equation is $F^*\geq 0$ in
$\R^N\times(0,\Tf]$ for supersolutions and, for subsolutions, $F_*\leq 0$ in $\R^N\times(0,\Tf]$
with the additional conditions $\F^k\leq 0$ on $\Man{k}$ where the ``tangential Hamiltonians'' $\F^k$
for $k=0..N$ are defined for $(x,t)\in \Man{k}$, $r\in \R$ and $p \in T_{(x,t)}\Man{k}$, by
$$
\F^k(x,t,u,p):=\sup_{\substack{(b,c,l)\in\BCL(x,t)\\ b\in T_{(x,t)}\Man{k}}}\big\{ -
    b\cdot p +cu - l\big\}\;.
$$
For $t=0$, we have analogous properties but for $\F_{init}$. We refer to
Chapter~\ref{chap:strat-def} for more precise definitions. 


Now we examine the needed assumptions on the $\BCL$ in $\R^N \times (0,\Tf]$ in order to have \TC and
\NCe: we are going to do it precisely for \TC and \NCe since, for \Mong, this is a more standard
consequence of \HBCL and we come back on that point in Chapter~\ref{chap:strat-def},  more
specifically in Section~\ref{sec:compstrat}. On the other hand, for $t=0$, such checking
is analogous using $\F_{init}$ and the associated Hamiltonians on $\Man{k}_0$. 

Since these assumptions are local and invariant by the $\Psi$-changes, we can state them in a ball
$B((x,t),r)$ centered at $(x,t) \in \Man{k}$ with a small radius $r>0$ and we can assume that, in
$B((x,t),r)$, $\M$ is an \TFS with $\Man{k}=(x,t) + V_k$, where $V_k$ is a $k$-dimensional vector
space in $\R^{N+1}$ and $B((x,t),r)$ intersects only $\Man{k}, \Man{k+1},\cdots, \Man{N+1}$. We
denote by $V_k^\bot$ the orthogonal space to $V_k$ and by $P^\bot$ the orthogonal projector on
$V_k^\bot$.  We trust the reader to be able to translate them for the original stratification and
$\BCL$.

In this framework, \TC \& \NCe are satisfied if, with the above notations
\index{Good framework for HJE!for control problems}
\label{page:GF.HJE:control}

\label{page:TCBCL} 
\index{Tangential continuity!control version}
\begin{assumption}{\TCBCL}{Tangential Continuity -- $\BCL$ version.}
For any  $0\leq k\leq N+1$ and for any $(x,t) \in
\Man{k}$, there exists a constant $C_1>0$ and a modulus $m:[0,+\infty)\to \R^+$ such that,
for any $j\geq  k$, if
$(y_1,t_1),(y_2,t_2)\in\Man{j}\cap B((x,t),r)$ with 
$(y_1,t_1)-(y_2,t_2) \in V_k$, then for any $(b_1,c_1,l_1) \in \BCL(y_1,t_1)$,
there exists $(b_2,c_2,l_2) \in \BCL(y_2,t_2)$ such that
$$ |b_1-b_2|\leq C_1(|y_1-y_2|+|t_1-t_2|)\quad,\quad 
|c_1-c_2|+|l_1-l_2| \leq m\big(|y_1-y_2|+|t_1-t_2|\big)\; .$$
\end{assumption}

\vspace*{-1.5em}

\index{Normal controllability!control version}
\label{page:NCBCL}\begin{assumption}{\NCBCL}{Normal Controllability -- $\BCL$ version.}
For any  $0\leq k\leq N+1$ and for any $(x,t) \in
\Man{k}$, there exists $\delta = \delta(x,t)>0$, such that, for any $(y,s) \in B((x,t),r)$,
one has $$B(0,\delta) \cap V_k^\bot \subset P^\bot\left(\B(y,s) \right)\, .$$
\end{assumption}

\vspace*{-1em}

Of course, the case $k=0$ is particular since $V_k=\{0\}$: here we impose
a complete controllability of the system in a neighborhood of $x\in\Man{0}$ since the
condition reduces to $B(0,\delta)\subset \B(y,t)$ because $V_k^\bot = \R^{N+1}$.

As we will see it throughout this book, the normal controllability assumption plays a key role in
all our analysis: first, at the control level, to obtain the viscosity subsolution inequalities for
the value function on each $\Man{k}$, then in the comparison proof to allow the regularization (in a
suitable sense) of the subsolutions and, last but not least, for the stability
result.\index{Regularization of subsolutions}

It is rather easy to prove that \NCBCL implies \NCe in this \TFS framework. We therefore concentrate
on \TCBCL and the following result first gives an important consequence of these assumptions:
the continuity of all the Hamiltonians $\{\F^k\}_{k=0..N}$, whose proof uses a combination of \TCBCL
and \NCBCL. We point out that, on the contrary, it is easy to prove that $\F^{N+1}$ satisfies \TC in $\Man{N+1}$.

With the same notations as above we set, for $(y,s) \in B((x,t),r)\cap \Man{k}$ 
$$\BCL^k(y,s):=\{(b,c,l)\in\BCL(y,s);\  b\in T_{(y,s)}\Man{k} = V_k\}\; ,$$ 
and $\B^k (y,s)$ is the set of all $b$ such that there exists $c,l$ for which
$(b,c,l)\in\BCL^k(y,s)$.

We have the
\begin{lemma}\label{tgfields} 
    If \TCBCL and \NCBCL hold, then
    \begin{enumerate}
    \item[$(i)$]  $\BCL^k (y,s)\neq \emptyset$ for any $(y,s) \in B((x,t),r)\cap \Man{k}$.
       
    \item[$(ii)$] There exists $\bar C_1>0$ and a modulus $\bar m$ such that, if
    $(y_1,t_1),(y_2,t_2)\in B((x,t),r)\cap \Man{k}$ and if $(b_1,c_1,l_1) \in \BCL^k (y_1,t_1)$,
    there exists $(b_2,c_2,l_2) \in \BCL^k(y_2,t_2)$ such that
    $$ |b_1-b_2|\leq \bar C_1(|y_1-y_2|+|t_1-t_2|)\;,\; |c_1-c_2|+|l_1-l_2| \leq 
    \bar m\big(|y_1-y_2|+|t_1-t_2|\big)\; .$$
    In particular, the Hamiltonian $\F^k$ satisfies \TC on $\M^k$, i.e. for any $R>0$, 
    for any $(y_1,t_1),(y_2,t_2)\in B((x,t),r)\cap \Man{k}$, $|r|\leq R$, $p\in V_k $ (or $p\in \R^{N+1}$)
    $$\begin{aligned}
        |\F^k(y_1,t_1,r,p)-\F^k(y_2,t_2,r,p)|\leq & \bar C_1(|y_1-y_2|+ |t_1-t_2|)|p|\\[2mm]
        & + (R+1)\bar m\big(|y_1-y_2|+|t_1-t_2|\big)\;.
    \end{aligned}
    $$
    
    \item[$(iii)$] For any $j\geq  k$, there exists $\tilde C_1>0$ and a modulus $\tilde m$ such that, if
    $(y_1,t_1),(y_2,t_2)\in\Man{j}\cap B(x,r)$ with $(y_1,t_1)-(y_2,t_2) \in V_k$, if $(b_1,c_1,l_1)
    \in \BCL^j (y_1,t_1)$, there exists $(b_2,c_2,l_2) \in \BCL^j(y_2,t_2)$ such that
    $$ |b_1-b_2|\leq \tilde C_1(|y_1-y_2|+|t_1-t_2|)\;,\; |c_1-c_2|+|l_1-l_2| 
    \leq \tilde m\big(|y_1-y_2|+|t_1-t_2|\big)\; .$$
    In particular, the Hamiltonian $\F^j$ satisfies \TC on $\M^j$, i.e. for any $R>0$, 
    for any $(y_1,t_1),(y_2,t_2)\in B((x,t),r)\cap \Man{k}$, $|r|\leq R$, $p\in V_k $ (or $p\in \R^{N+1}$)
    $$\begin{aligned}
        |\F^j(y_1,t_1,r,p)-\F^j(y_2,t_2,r,p)|\leq & \tilde C_1(|y_1-y_2|+|t_1-t_2|)|p|\\[2mm]
        & + (R+1)\tilde m\big(|y_1-y_2|+|t_1-t_2|\big)\;.
    \end{aligned}$$
    \end{enumerate}
\end{lemma}

\begin{proof}
    The first part of the result is a direct consequence of \NCBCL: indeed $0\in P^\bot\left(\B(y,s)
    \right)$, hence there exists $(b,c,l)\in \BCL (y,s)$ such that $P^\bot (b)=0$, i.e. $b \in V_k =
    T_{(y,s)}\Man{k}$.

    For the second part of the result, we use \TCBCL: if $(b_1,c_1,l_1) \in \BCL^k(y_1,t_1) \subset
    \BCL(y_1,t_1) $, there exists $(b_2,c_2,l_2) \in \BCL(y_2,t_2)$ such that
    $$ |b_1-b_2|\leq C_1(|y_1-y_2|+|t_1-t_2|)\quad,\quad 
    |c_1-c_2|+|l_1-l_2| \leq m\big(|y_1-y_2|+|t_1-t_2|\big)\; .$$
    We have to modify $(b_2,c_2,l_2)$ in order to obtain $(\tilde b_2,\tilde c_2,\tilde l_2)\in
    \BCL^k(y_2,t_2)$ with the right property. To do so, we notice that, since $P^\bot(b_1)=0$ then
    $|P^\bot(b_2)| \leq \eta: =C_1(|y_1-y_2|+|t_1-t_2|)$. 

    If $P^\bot(b_2)=0$ the result holds, hence we may assume that $P^\bot(b_2)\neq 0$ and set
    $$ e= \frac{P^\bot(b_2)}{|P^\bot(b_2)|}\; .$$
    Using \NCBCL, there exists $(\bar b_2,\bar c_2,\bar l_2)\in \BCL(y_2,t_2)$ such that
    $P^\bot(\bar b_2) = -(\delta/2)e$ and we consider the convex combination 
    $$ (\tilde b_2,\tilde c_2,\tilde l_2):= (1-\alpha)(b_2,c_2,l_2) + \alpha(\bar b_2,\bar c_2,\bar
    l_2)\; .$$ 
    Since
    $$P^\bot(\tilde b_2)=(1-\alpha)P^\bot(b_2) + \alpha P^\bot(\bar b_2)=(1-\alpha)\eta e 
    - \frac{\delta}2 \alpha e\;,$$
    choosing $\alpha = \eta/(\eta +\delta/2)$ we get $P^\bot(\tilde b_2)=0$. Therefore $(\tilde
    b_2,\tilde c_2,\tilde l_2)\in \BCL^k(y_2,t_2)$ and the estimates on $|b_1-\tilde b_2|$,
    $|c_1-\tilde c_2|$, $|l_1-\tilde l_2|$ are an easy consequence of the value of $\alpha$, because
    of the definition of $\eta$ and the properties of $b_2,c_2,l_2$. Indeed, the difference between
    $(\tilde b_2, \tilde c_2,\tilde l_2)$ and $(b_2,c_2,l_2)$ behaves like $3M\alpha \leq
    3M\delta^{-1} \eta$ and therefore the result holds with 
    $$ \bar C_1:= (1+3M\delta^{-1})C_1 \quad\hbox{and}\quad \bar m(\tau)= m(\tau)+ 3M\delta^{-1}
    C_1\tau\; .$$

    Finally the \TC inequality for $\F^k$ is a direct consequence of the previous result. The
    third result follows from analogous arguments as in $(ii)$.  
\end{proof}

\subsection{Ishii solutions for a codimension one discontinuous Hamilton-Jacobi Equation}
\label{sec:GA-codim1}

We conclude this section by some remarks on the model problem which is studied in Part~\ref{part:codim1} and~\ref{part:NA} where $\OO=\R^N\times (0,\Tf)$, $X=(x,t)$ and
$$
\G(x,t,r,(p_x,p_t)):=\begin{cases}p_t +H_1 (x,t,r,p_x) & \hbox{if  }x_N >0,\\
p_t +H_2 (x,t,r,p_x) & \hbox{if  }x_N <0.
\end{cases}$$
For Part~\ref{part:codim1}, we are in the control case and we use the standard Ishii inequalities, namely $G^*\geq 0$ in $\OO$ and 
$G_*\leq 0$ in $\OO$. We can use \TCBCL and \NCBCL which are satisfied if \HBAHJ holds and if $H_1,H_2$ satisfies the 
assumption \NCoH (see p.~\pageref{page:NCoH}) on $\Man{N}=\H:=\{x:\ x_N=0\}\times (0,\Tf)$ . In fact, \Mong but also \LOCa, \LOCb are 
also satisfied under these assumptions.

Concerning Part~\ref{part:NA}, we point out that essentially the same type of assumptions are needed but since the Hamiltonians
$H_1,H_2$ will only be assumed to be quasi-convex, we have to come back to the \TC, \NCe formulations.


\chapter{Other Tools}
\fancyhead[CO]{HJ-Equations with Discontinuities: Other Tools}
\label{chap:other.tools}

\abstract{In this catch-all chapter are gathered several results which are of either of general
interest like those on semi-convex/semi-concave functions and on penalizations; or related to
networks like those for quasi-convex functions or the Kirchhoff-related lemma.}

\section{Semi-convex and semi-concave functions: the main properties}
\label{sect:semi.convexity}

\index{Semi-convex, semi-concave functions}
The aim of this section is to describe the properties of semi-convex and semi-concave functions which
will be used throughout this book, in particular those connected to their differentiability. Considering
Section~\ref{sec:regplus-subsol}, it is clear that we are not going to manipulate functions which are 
semi-convex or semi-concave \wrt all variables but only in the ``tangential variables''; anyway, since this latter case
consists only in applying the results of the first one by fixing the normal coordinates, we will only be interested in this section in
the case of the functions which are semi-convex/semi-concave \wrt all variables.

We first recall that, if $\OO \subset \R^N$ is a convex domain and $f:\OO \to \R$, the function $f$ is semi-convex
\resp{semi-concave} if there exists a constant $C\geq 0$ such that $x\mapsto f(x)+C|x|^2$ is convex
\resp{$x\mapsto f(x)-C|x|^2$ is concave}.

In the sequel, we consider only the semi-convex case, the semi-concave one being deduced by changing $f$ in $-f$ in the results
below. In addition, we point out that all the properties we are going to describe are nothing but properties of convex functions which
are translated in a suitable (and easy) way, the term $C|x|^2$ being smooth and therefore causing no problem for the differentiability.

We list all the properties in the following result
\begin{proposition}\label{prop:semi-conv-prop}\emph{--- Properties of semi-convex/semi-concave
    functions.}\smsp
    If $f:\OO \to \R$ is a locally bounded functioni, semi-convex for a constant $C\geq 0$, then
\begin{enumerate}
    \item[$(i)$] $f$ is locally Lipschitz continuous in $\OO$ and if $\overline{B(x,2r)}\subset \OO$, the Lipschitz constant of $f$ 
in $\overline{B(x,r)}$ depends only on $||f||_{L^\infty(\overline{B(x,r)})}$.
\item[$(ii)$] $f$ is differentiable a.e. in $\OO$.
\item[$(iii)$] For any $x \in \OO$, $D^-_{\OO} f(x)\neq \emptyset$ and if $p \in D^-_{\OO} f(x)$, we have, for all $y \in \OO$,
\begin{equation}\label{eq:semic-prop}
f(y) \geq f(x) +p\cdot (y-x) -2C|y-x|^2\; .
\end{equation}
\item[$(iv)$] Let $(f_\e)_\e$ be a sequence of functions which are semi-convex with the same constant $C$ and which are converging to
$f$ locally uniformly in $\OO$ and let $(\xe)_\e$ a sequence of points of $\OO$ which converges to $x\in \OO$. If $\pe \in D^-_{\OO} f(\xe)$
and if $(p_{\e'})_{\e'}$ is subsequence of $(\pe)_\e$ which converges to $p$ then
$p\in D^-_{\OO} f(x)$. In particular, if $f_\e$ is differentiable at $\xe$ for any $\e$ and if $f$ is differentiable at $x$, then $Df_\e (\xe)\to Df(x)$.
\item[$(v)$] If $\varphi$ is either a $C^1$ or a semi-concave function defined on $\OO$ and if $x$ is a maximum point of $f-\varphi$, then $f$ is differentiable at $x$, $\varphi$ is also differentiable at $x$ in the semi-concave case and $Df(x)=D\varphi(x)$.
\end{enumerate}
\end{proposition}

Of course we are not going to give a complete proof of Proposition~\ref{prop:semi-conv-prop}: as we mentioned it above, most of the
results are very classical for convex functions and extend without any difficulty to the case of semi-convex ones. But we provide some comments for 
each of them.

\begin{enumerate}
\item $(i)$ and $(ii)$ are famous classical results for convex functions, $(ii)$ being a consequence
    of $(i)$ through Rademacher's Theorem (even
if historically Rademacher's Theorem is more a consequence of $(ii)$).
\item $(iii)$ also reflects a classical property of convex function, in particular Inequality~\eqref{eq:semic-prop} with the correcting term $-2C|y-x|
^2$.
\item $(iv)$ is an easy consequence of  Inequality~\eqref{eq:semic-prop}. We point out that the existence of converging subsequences $
(p_{\e'})_{\e'}$ is a consequence of $(i)$ since it is easy to show that $|\pe|$ is controlled by the Lipschitz constant of $f_\e$ and these Lipschitz 
constants are uniformly bounded by $(i)$ and the local uniform convergence of the sequence $(f_\e)_\e$. An interesting particular case is the 
choice when $f_\e\equiv f$ where 
we have some kind of ``continuity of the gradient'' since, if we have a sequence $(\xe)_\e$ of points where $f$ is differentiable
which converges to $x\in \OO$ where $f$ is differentiable, then $Df (\xe)\to Df(x)$. This is proved by a standard compactness argument 
since $Df(x)$ is the only possible limit of subsequences of  $(Df (\xe))_\e$.
\item Property $(v)$ will play a key role for us since we are going to be any time in this context (we recall here that this will be only a property to be used in the ``tangential variables''). This property is a consequence of the following result: if $D^-_{\OO} f(x)\neq \emptyset$ AND $D^+_{\OO} f(x)\neq \emptyset$ then $f$ is differentiable at $x$ and $D^-_{\OO} f(x)=D^+_{\OO} f(x)=\{Df(x)\}$. In our context, we know by 
$(iii)$ that $D^-_{\OO} f(x)\neq \emptyset$ and then we have two cases

-- if $\varphi$ is $C^1$, the maximum point property implies $D\varphi (x) \in D^+_{\OO} f(x)$ which is therefore non-empty and the conclusion follows readily.

-- If $\varphi$ is semi-concave, then $D^+_{\OO} \varphi (x)\neq \emptyset$ by an analogous property
        of $(iii)$ for semi-concave function and
the maximum point property both implies $D^+_{\OO} \varphi (x)\subset D^+_{\OO} f (x)$ and $D^-_{\OO} f (x)\subset D^-_{\OO} \varphi (x)$. Hence both $f$ and $\varphi$ are differentiable at $x$ and $Df(x)=D\varphi(x)$.
\end{enumerate}

\begin{remark}\label{rem:cont-SC-D} Property $(iv)$ will mainly be used in the case when $f$ is differentiable
at $x$. Then, for any sequence $(\xe)_\e$ of points of $\OO$ which converges to $x\in \OO$ and for any choice of
$\pe \in D^-_{\OO} f_\e (\xe)$, the sequence of $(\pe)_\e$ converges to $Df(x)$. Indeed, the sequence $(\pe)_\e$ is bounded,
hence it lies in a compact subset of $\R^N$ and $Df(x)$ is the only possible limit for converging subsequences of $(\pe)_\e$.
\end{remark}

\section{Quasi-convexity: definition and main properties}
\label{sect:quasi.convexity}

\index{Quasi-convex functions}\label{page:QCR}
Let $\mC\subset\R^N$ be a convex set. A \emph{quasi-convex} function $f:\mC\to\R$ is a function
such that, for any $a\in\R$, the lower level set $\{x:f(x)\leq a\}$ is convex. 

An equivalent definition is: for any $x,y\in \mC$ and $\lambda\in(0,1)$, $$f(\lambda x+(1-\lambda)y)\leq
\max\{f(x),f(y)\}\;.$$

Of course, convex functions are quasi-convex but the converse is false since quasi-convex functions
can be discontinuous, even if they are bounded: for example, take, in $\R^N$, the indicator function
of the complementary of a convex set. Hence, one of the differences between convex and quasi-convex
functions is that quasi-convex functions may have various ``flat'' zones, not only where they
achieve their minimum.
\subsection{Quasi-convex functions on the real line}

We introduce the assumption

\begin{assumption}{\QCR}{Basic quasi-convexity assumption.}
    The function $f:\R\to\R$ is continuous, coercive and quasi-convex.  
\end{assumption}

The first (classical) result we have for such functions is the
\begin{lemma} If $f:\R\to\R$ satisfies \QCR then
    \begin{enumerate}
        \item[$(i)$] there exists $m^-(f)\leq m_+(f)$ such that the set where $f$ achieves its
            minimum is exactly the interval $[m^-(f), m_+(f)]$.
    
        \item[$(ii)$] $f$ is nonincreasing on $]-\infty,m^-(f)[$ and
        nondecreasing on $]m_+(f),+\infty[$.
    
    \item[$(iii)$] $f=\max\{f^\sharp,f_\flat\}$ where $f^\sharp$ is
    nondecreasing and $f_\flat$ is nonincreasing. 
    \end{enumerate}
\end{lemma}

\begin{proof} 
    The proof of $(i)$ is easy: since $f$ is continuous and coercive, it is bounded from below and
    achieves its minimum.  Moreover by quasi-convexity, the set $\{x:f(x)\leq \min_\R (f)\}$ is
    convex, hence this is an interval $[m^-(f), m_+(f)]$.

    For $(ii)$, we consider $x,y \in ]-\infty,m^-(f)[$ with $x<y$. If $f(x)<f(y)$, then, by the
    quasi-convexity of $f$, the convex set $\{t:\ f(t)\leq f(x)\}$ contains $x$ and $m^-(f)$, hence
    all the interval $[x,m^-(f)]$. A contradiction since $y\in [x,m^-(f)]$. Hence $f$ is
    nonincreasing on $]-\infty,m^-(f)[$ and an analogous proof shows that $f$ nondecreasing on
    $]m_+(f),+\infty[$.

    For $(iii)$, we consider $$f^\sharp(x)= \min\{f(t); t\geq x\}\; ,\; f_\flat(x)=\min\{f(t); t\leq
    x\}.$$ Clearly we have
    $$ f^\sharp(x)= \min_\R (f) \; \hbox{if  }x\leq m_+(f)\quad ,\quad f_\flat(x)=\min_\R (f)\; 
    \hbox{if  }x\geq m^-(f)\; ,$$
    while, by using $(ii)$,
    $$ f^\sharp(x)= f(x) \; \hbox{if  }x> m_+(f)\quad ,\quad f_\flat(x)=\min_\R (f)\; 
    \hbox{if  }x< m^-(f)\; .$$
    The conclusion follows by analyzing the different cases $x< m^-(f)$, $m^-(f)\leq x\leq m_+(f)$
    and $x> m_+(f)$.  
\end{proof}

\subsection{On the maximum of two quasi-convex functions}\label{sec:max-qc}

In this section, we describe a result which is crucial in order to give sufficient conditions for
the uniqueness of Ishii solutions in problems with codimension 1 discontinuities (see
Section~\ref{upoH}). 

Let $f,g:\R\to\R$ satisfy \QCR and define
$$M(s):=\max\{f(s),g(s)\}\;,\qquad M^\reg(s):=\max\{f^\sharp(s),g_\flat(s)\}\;.$$
We point out that we use the strange notation $M^\reg$ to be consistent with Section~\ref{upoH}. Notice
that the definition of $M^\reg$ is not symmetric on $f$ and $g$.

\begin{lemma}\label{lem:qcr.Mreg}
    We assume that $f,g$ satisfy \QCR.
    There exists $\nu_1\leq\nu_2$ such that 
    $$M^\reg(s):=\begin{cases}g_\flat(s)>f^\sharp(s) & \text{if }s< \nu_1\;,\\
        f^\sharp(s)=g_\flat(s) & \text{if }\nu_1\leq s\leq \nu_2\;,\\
        f^\sharp(s)>g_\flat(s) & \text{if }s> \nu_2\;.
    \end{cases}$$
    Of course, $\min\limits_{s\in\R}M^\reg(s)$ is attained on $[\nu_1,\nu_2]$.
\end{lemma}
\begin{proof}
    We introduce the function $\varphi(s):=f^\sharp(s)-g_\flat(s)$. Due to the properties of
    $f^\sharp$ and $b_\flat$, the function $\varphi$ is nondecreasing. Moreover, due to the
    coercivity assumption,  $\varphi(s)\to-\infty$ as
    $x\to-\infty$ and $\varphi(s)\to+\infty$ as $x\to+\infty$. Therefore, there exists $\nu_1\leq
    \nu_2$ such that $\varphi(s)<0$ if $s<\nu_1$, $\varphi(s)>0$ if $s>\nu_2$ and $\varphi(s)=0$ on
    $[\nu_1,\nu,_2]$. The lemma directly follows.
\end{proof}

\begin{proposition}\label{prop:quasi.convex.maxreg}
    Let $f,g:\R\to\R$ satisfy \QCR. If $m_+(f)\leq m^-(g)$ then the following
    property holds
    $$\min_{s\in\R}M(s)=\min_{s\in\R}M^\reg(s)\;.$$
\end{proposition} 

\begin{proof}
    Notice first that of course the inequality $\max\{f,g\}\geq \max\{f^\sharp,g_\flat\}$ holds
    simply because of the definition of $f^\sharp$ and $g_\flat$; therefore the same inequality
    holds when taking the minimum over $s$. 

    In order to get the opposite inequality, we first remark that, by Lemma~\ref{lem:qcr.Mreg}, the
    minimum of $M^\reg$ is attained at some point $s_0$ which satisfies $s_0\in[\nu_1,\nu_2]$. Moreover,
    $f^\sharp(s_0)=g_\flat(s_0)$. There are three cases, some of which may be void.

    \noindent \emph{First case:} $s_0\in[m_+(f),m^-(g)]$. In this case
    the conclusion easily follows from the fact that $f^\sharp(s_0)=f(s_0)=g_\flat(s_0)=g(s_0)$: we
    deduce immediately that $\min_\R (M^\reg)=M^\reg(s_0)=M(s_0)\geq \min_\R (M)$.

    \noindent \emph{Second case:} $s_0\leq m_+(f)\leq m^-(g)$. This implies that
    $f^\sharp(s_0)=\min_\R (f)=g_\flat(s_0)$ and $\min_\R(M^\reg)=M^\reg(s_0)=\min_\R(f)$. 

    Considering the situation at $s=m_+(f)$ we see that
    $$\begin{aligned}
        g(m_+(f)) &= g_\flat(m_+(f))\quad \text{because $m_+(f)\leq m^-(g)$}\\
        &\leq g_\flat(s_0)\quad \text{because $g_\flat$ is nonincreasing}\\
        &\leq f^\sharp(s_0)\quad \text{by the definition of $s_0$}\\
        &\leq f^\sharp(m_+(f))\quad \text{because $f^\sharp$ is flat for $s\leq m_+(f)$}\\
        &\leq f(m_+(f))=\min_\R(f)\;.
    \end{aligned}$$
    We deduce that, at $s=m_+(f)$, $M(m_+(f))=\min_\R(f)=\min_\R(M^\reg)$.
    Hence, we conclude that $\min_\R(M^\reg)\geq\min_\R(M)$.

    \noindent \emph{Third case:} if $s_0\geq m^-(g)\geq m_+(f)$, the proof is the same
    after reversing the roles of $f^\sharp$ and $g_\flat$. 

    The conclusion is that, in any case, $\min_\R(M^\reg)\geq\min_\R(M)$ which implies that those
    minima are equal.
\end{proof}

\subsection{Application to quasi-convex Hamiltonians}\label{sec:QCH-def}

As we have seen in the previous sections, throughout this book we deal with Hamiltonians of the form
$H(x,t,r,p)$. Those may be either convex, Lipschitz, or have a quasi-convexity property that we
describe now. 

The quasi-convex case (mainly exposed in Part~\ref{part:NA}) is defined in the following way: if we
set $p=(p',p_N)$ with $p'\in \R^{N-1}$ and $p_N \in \R$, we will say that we are in the {\em
quasi-convex case} if

\label{page:HQC}
\begin{assumption}{\HQC}{Quasi-convex Hamiltonians.}
    For any  $(x,t,r,p')$, the function $h:s\mapsto H(x,t,r,p'+se_N)$ satisfies \QCR.
\end{assumption}

Using the previous sections, we can introduce the new Hamiltonians
$$ H^- (x,t,r,p) =h^\sharp(p_N)=\big[ H(x,t,r,p'+p_N e_N)\Big]^\sharp\;,$$
$$ H^+ (x,t,r,p)=h_\flat(s)=\big[ H(x,t,r,p'+s e_N)\Big]_\flat\;.$$
Thanks to the above results, we have $H=\max(H^+,H^-)$. We use extensively this decomposition in
Part~\ref{part:NA} and we point out that, if $H$ satisfies \hyp{BA-HJ}, then the Hamiltonians
$H^+,H^-$ also satisfy \hyp{BA-HJ}.

\section{A strange, Kirchhoff-related lemma}\label{tmgktc}

In Part~\ref{part:NA}, the following lemma will be useful in order to connect general Kirchhoff type
conditions with flux-limited type conditions on the interface.

\begin{lemma}\label{lem:fgh}
Assume that f,g : $\R \to \R$ and $h:\R^2 \to \R$ are continuous functions such that
\begin{enumerate}
\item[$(i)$] $f$ is an increasing function with $f(t) \to +\infty$ as $t\to +\infty$,
\item[$(ii)$] $g$ is a decreasing function with $g(t) \to +\infty$ as $t\to -\infty$,
\item[$(iii)$] there exists $\alpha >0$ such that, for any $t_2\geq t_1$ and $s_2 \leq s_1$, we have
$$ h(t_2,s_2)-h(t_1,s_1)
 \leq -\alpha(t_2-t_1) +\alpha(s_2-s_1)\; .$$ 
 \end{enumerate}
If $\psi : \R^2 \to \R$ is the function defined by
$$ \psi(t,s) :=\max(f(t), g(s), h(t,s))\; ,$$
then $\psi$ is a coercive continuous function in $\R^2$ and there exists $(\tb,\ovs)$ such that
\begin{equation}\label{min-prop}
\psi(\tb,\ovs)=\min_{t,s}\,(\psi(t,s))
\end{equation}
and 
\begin{equation}\label{min-prop-cons}
f(\tb)=g(\ovs)=h(\tb,\ovs)\; .
\end{equation}
Moreover, if a point $(\tilde t,\tilde s)\in \R^2$ satisfies \eqref{min-prop-cons} then $(\tilde
t,\tilde s) $ is a minimum point of $\psi$.  Finally,
$$ \min_{t,s}\left\{\max(f(t), g(s), h(t,s))\right\}=\max_{t,s}\left\{\min(f(t), g(s), h(t,s))\right\}.$$
\end{lemma}

In the statement of the above lemma, we point out that the assumption on $h$ implies that $h(t,s)$ is a strictly decreasing function of
$t$ and a strictly increasing function of $s$ with $h(t,s) \to +\infty$ if $t\to -\infty$, $s$ remaining bounded or if $s\to +\infty$, $t$ remaining
bounded.

\begin{proof} Using the three properties we impose on $f,g,h$, and in particular, the consequences of the assumption on $h$
we describe above, it is easy to prove that $\psi$ is actually continuous and
coercive; therefore such a minimum point $(\tb,\ovs)$ exists. 

We have to show that \eqref{min-prop-cons} holds and to do so, we may assume without loss of generality that $f$ is
strictly increasing and $g$ is strictly decreasing. Otherwise, we may prove the result for $f(t) +\e t$ and $g(s)-\e s$ for $\e>0$ and pass
to the limit $\e \to 0$ remarking that the associated minimum points remain in a fixed compact subset of $\R^2$.

If $m=\min_{t,s}\,(\psi(t,s))$, we first notice that $h(\tb,\ovs)=m$. Otherwise $h(\tb,\ovs)<m$ and it is clear enough by using the monotonicity of $f$ and $g$ that, for 
$\delta >0$ small enough, then
$$ \psi(\tb-\delta,\ovs+\delta ) < \psi(\tb,\ovs)\; ,$$
a contradiction.

In the same way, if $f(\tb)<m$, using the properties of $h$, there exists $\delta,\delta' >0$ small enough such that $h(\tb+\delta,\ovs+\delta')<m$, 
$g(\ovs+\delta')<m$ and $\psi(\tb+\delta,\ovs+\delta') < \psi(\tb,\ovs)$, again a contradiction.

A similar proof allowing to conclude that $g(\ovs)=m$, \eqref{min-prop-cons} holds.

Notice that if we have replaced $f(t)$ by
$f(t) +\e t$ and $g(t)$ by $g(s)-\e s$, we can let $\e$ tend to $0$ and keep this property for at least one minimum point.

Now we consider a point $(\tilde t,\tilde s)\in \R^2$ which satisfies \eqref{min-prop-cons} 
and we pick any point $(t,s) \in \R^2$. We examine the different possible cases, 
taking into account the particular form of $\psi$ and the monotonicity properties of $f,g,h$, 
using that, of course, $\psi(\tilde t, \tilde s)= f(\tilde t)=g(\tilde s)=h(\tilde t,\tilde s)$
\begin{enumerate}
\item If $t\geq \tilde t$, $\psi(t,s) \geq f(t) \geq  f(\tilde t)=\psi(\tilde t, \tilde s)$.
\item If $s \leq \tilde s$, the same conclusion holds by using that $g$ is decreasing.
\item If $t\leq \tilde t$ and $s \geq \tilde s$, then $\psi(t,s) \geq h(t,s) \geq  h(\tilde t, \tilde s)=\psi(\tilde t, \tilde s)$.
\end{enumerate}
And the conclusion follows since we have obtained that $\psi$ reaches its minimum at $(\tilde t,\tilde s)$.

For the last property, we set
$$\chi(t,s)=\min(f(t), g(s), h(t,s))\; .$$
If, as above, $(\tilde t,\tilde s)\in \R^2$ is a point which satisfies \eqref{min-prop-cons}, we have $\chi(\tilde t, \tilde s)= f(\tilde t)=g(\tilde s)=h(\tilde t,\tilde 
s)$ and by similar arguments as above
\begin{enumerate}
\item If $t\leq \tilde t$, $\chi(t,s) \leq f(t) \leq  f(\tilde t)=\chi(\tilde t, \tilde s)$.
\item If $s \geq \tilde s$, the same conclusion holds by using that $g$ is decreasing.
\item If $t\geq \tilde t$ and $s \leq \tilde s$, then $\chi(t,s) \leq h(t,s) \leq  h(\tilde t, \tilde s)=\chi(\tilde t, \tilde s)$.
\end{enumerate}
And the proof is complete.
\end{proof}

\begin{remark}\label{min-max} A similar result to the last part of Lemma~\ref{lem:fgh}, but with a simpler proof,  is
\begin{equation}\label{eq:min-max}
 \min_t \{\max(f(t), g(t))\}=\max_t \{\min(f(t),g(t))\}\; .
\end{equation}
This equality is also useful in Part~\ref{part:NA}.
\end{remark}

\section{A few results for penalized problems}\label{sec:cv-pen}

In viscosity solutions' theory, several proofs require penalization arguments, i.e. approximations of
maxima or minima by penalizing the function. The most emblematic example is certainly the
doubling of variables in comparison proofs but there are several other examples, such as the
treatment of some boundary conditions (evolution equations set in $(0,\Tf)$ which hold up to time $\Tf$
or more generally boundary conditions in the case when all dynamics are pointing inward the domain)
or the convergence of regularization by inf or sup-convolution...etc.

Instead of referring to these (rather easy) results as ``standard results'' all along this book, we
have decided to provide two general lemmas gathering the key informations, one for penalization in
compact sets, the other one (more restrictive) concerns the penalization at infinity.

\subsection{The compact case}

\index{Penalization!the compact case}
\begin{lemma}\label{lem:cv-pen}\emph{--- Penalization procedure, the compact case.}\smsp 
    Let $w:K\to\R$ be an \usc function defined on some compact set $K\subset\R^p$ and $F\subset K$
    be closed. We denote by $M:=\max_{z\in F}w(z)$. For any $\e>0$ let $\chi_\e:K\to\R\cup\{+\infty\}$ 
    satisfying
\begin{enumerate}
    \item[$(i)$] the functions $\{\chi_\e\}$ are uniformly bounded from below
        and \lsc \footnote{in the expected generalized sense in order to take into account the
    $+\infty$ value at some points if necessary.}; 
    
    \item[$(ii)$] $\limiinf \chi_\e  (z)=
        \begin{cases} 0 & \hbox{if $z\in
        F$\;,}\\ +\infty & \hbox{if $z \in K\setminus F$}\,;\end{cases}$

    \item[$(iii)$] for any $z_0 \in F$, there exists $(z_0^\e)_\e$ such that
        $w(z_0^\e)-\chi_\e(z_0^\e)\to w(z_0)$ as $\e\to 0$.
\end{enumerate}
Then
\begin{enumerate}
    \item[$1.$] $M_\e:=\max\limits_{z\in K}\left(w(z)-\chi_\e(z)\right)\to M$ as $\e\to 0$.

    \item[$2.$] For any $\e>0$ let $\ze$ be a maximum point of $z\mapsto w(z)-\chi_\e(z)$. If
        $(z_{\e'})_{\e'}$ is a subsequence of $(z_\e)$ converging to some $\zb$, then
        $$\zb\in F\;,\; w(\zb)=M\; ,\; w(z_{\e'})\to w(\zb)\; ,\; \chi_{\e'} (z_{\e'}) \to 0\; .$$

    \item[$3.$] If $w=w_1-w_2$ where $w_1$ is \usc and $w_2$ is \lsc, then $w_1(z_{\e'})\to w_1(\zb)$ 
        and $w_2(z_{\e'})\to w_2(\zb)$.

    \item[$4.$] If there is a unique maximum point $\zb$ of $w$ on $F$ then $\ze \to \zb$,
        $w(\ze)\to w(\zb)$ and $\chi_\e (\ze) \to 0$.
\end{enumerate}
\end{lemma}

\begin{proof} 
    Since $K$ is compact and $F$ is a closed subset of $K$, there exists $z_0$ such that $w(z_0)=M$.
    By the definition of $\ze$ and $(iii)$, we have $$ M+o_\e(1)=w(z_0^\e)-\chi_\e(z_0^\e) \leq
    w(\ze)-\chi_\e (\ze)=M_\e$$
    and this inequality immediately gives $\limsup M_\e \geq M$. 

    On the other hand, if we extract a converging subsequence $z_{\e'}\to \zb\in K$, by
    letting $\e'\to 0$ and using the upper semicontinuity of $w$ we obtain
    $$M \leq \liminf\left( w(z_{\e'})-
    \chi_{\e'}(z_{\e'})\right) \leq \limsup\left( w(z_{\e'})- \chi_{\e'}(z_{\e'})\right)\leq
    w(\zb)-\limiinf \chi_\e(\zb)\; .$$
    Using $(i)$, we see that necessarily $\zb\in F$ since $\limiinf \chi_\e(\zb)$ cannot be
    $+\infty$, therefore $\limiinf \chi_\e(\zb)=0$. We deduce from this property and the above
    inequality that $w(\zb)\geq M$ but since $\zb \in F$, we conclude that $w(\zb)=M$.

    Gathering all these informations, the above inequality can be rewritten as
    $$M \leq \liminf\left( w(z_{\e'})- \chi_{\e'}(z_{\e'})\right) \leq \limsup\left( w(z_{\e'})-
    \chi_{\e'}(z_{\e'})\right)\leq M\; ,$$
    and therefore $M_{\e'} =w(z_{\e'})- \chi_{\e'}(z_{\e'})\to M$. 

    Extracting first a subsequence such that $\lim M_{\e'}=\liminf M_\e$ and then a converging
    subsequence out of $(z_{\e'})_{\e'}$, the above argument shows that $\liminf M_\e=M$ and
    therefore $M_\e \to M$. This proves 1.

    Point 2 is a direct consequence of the above argument: for any converging subsequence
    $z_{\e'}\to \zb\in K$, we have $\zb \in F$, $w(\zb)=M$ and since $\limsup w(z_{\e'}) \leq
    w(\zb)\leq M$ by the upper semi-continuity of $w$ and $\liminf \chi_{\e'}(z_{\e'})\geq \limiinf
    \chi_\e(\zb)=0$, the only possibility to have such a convergence to $M$ is $w(z_{\e'})\to M=
    w(\zb)$ and $\lim \chi_{\e'}(z_{\e'}) = 0$.

    For Point 3, the argument is analogous: since $\limsup w_1(z_{\e'})\leq w_1(\zb)$ and $\liminf
    w_2(z_{\e'})\geq w_2(\zb)$, the only possibility to have $w(z_{\e'})\to w(\zb)$ is to have at
    the same time $w_1(z_{\e'})\to w_1(\zb)$ and $w_2(z_{\e'})\to w_2(\zb)$.

    Finally 4. comes from a standard compactness argument.
\end{proof}

\noindent\textbf{Typical application: the doubling of variables --}
After the localization procedure described in Section~\ref{sect:htc}, we get two functions
$u,v:\overline{B(x,r)}\to \R$ for some $x\in \R^N$ and $r>0$, $u$ being \usc while $v$ is \lsc and
we are considering $M:=\max_{x\in \overline{B(x,r)}} (u(x)-v(x))$, that we approximate
by the maximum of the function
$$ \psi_\e(x,y)= u(x)-v(y)-\frac{|x-y|^2}{\eps^2}\; .$$
We apply Lemma~\ref{lem:cv-pen} with $K=\overline{B(x,r)}\times \overline{B(x,r)}$, $F=K\cap
\{(x,y):\ x=y\}$, $$z=(x,y)\;,\; w(x,y)=u(x)-v(y)\;,\; \chi_\e(x,y)=\dfrac{|x-y|^2}{\eps^2}$$ and $w_1=u$,
$w_2=v$. We notice that Assumptions~$(i)-(ii)-(iii)$ for $\chi_\e$ are obviously satisfied with
$z_0^\e=z_0$ for any $\e$.
 
So, if $(\xe,\ye) \in K$ is a maximum point of $\psi_\e$ in $K$ and if $(x_{\e'},y_{\e'})$ 
 is a converging subsequence of maximum points of $\psi_{\e'}$, we first have that
 $(x_{\e'},y_{\e'})\to (\xb,\xb) \in F$ and $$ u(x_{\e'})\to u(\xb)\; ,\; v(y_{\e'})\to v(\xb)\; ,\;
 \frac{|x_{\e'}-y_{\e'}|^2}{(\eps')^2}\to 0\; ,$$
 which is the classical result we use.

 \

 \noindent\textbf{Remarks on the assumptions}

\noindent\textbf{(a)}
As a first comment, we point out that, one way or the
 other, the ``compactness'' assumption on $K$ in Lemma~\ref{lem:cv-pen} is necessary, although it
 may be replaced by a stronger assumption on $w$ like coercivity which prevents infinity to play a
 role, see Subsection~\ref{subsec:pen.inf} below.

 Moreover, this type of lemma does not hold in non-compact situations, in general, even if we
 replace $\max$ by $\sup$. Indeed if we look at the following penalization $$ \psi_\e(x,y)=
 \sin(x^2)-\sin(y^2)-\frac{|x-y|^2}{\eps^2}-\e |x|\; ,$$ but with $K=\R\times \R$ and $F=\{(x,y):\
 x=y\}$, the reader will easily check, using the non-uniform continuity of $\sin(x^2)$, that $M_\e$
 exists and $M_\e \to 2$ as $\e\to 0$ while $M=\sup_{(x,y)\in F}(\sin(x^2)-\sin(y^2))=0$.

 \

\noindent\textbf{(b)}
Notice that $\chi_\e$ can take the value $+\infty$, a case which gives important applications too.
For instance if $K=[0,\Tf]$, we can handle terms like $\e/(\Tf-t)$ in $\chi_\e$, which prevent the
maximum to be attained at $t=\Tf$. The lower semicontinuity
property for $\chi_\e$ holds since $$ \lim_{\substack{ t\to \Tf\\ t<\Tf}} \chi_\e (t)=+\infty\; .$$
Similarly if $\Omega$ is a bounded smooth domain, we can use a penalization like
$\e[d(x)]^{-1}$ in $\chi_\e:\Omegb\to \R\cup \{+\infty\}$ where $d(\cdot)$ stands for the
distance to the boundary of $\Omega$. Such penalizations avoid maximum points at the boundary 
---See for instance Proposition~\ref{sub-up-to-b} where this approached is used.

 \

\noindent\textbf{(c)}
Finally, let us explain the (admittedly strange) Assumption~$(iii)$ for $\chi_\e$. In
state-constrained problems where the subsolution inequalities hold only in a domain $\Omega$ while
the supersolution ones hold on $\overline \Omega$, one needs to ``push inside $\Omega$'' the point
$x$ corresponding to the subsolution. In order to prove comparison result for such problems,
Soner \cite{Son1,Son2} introduces penalization terms of the form 
$$ \left\vert \frac{x-y}{\eps}+n(y)\right\vert^2$$
where, if $\domeg$ is smooth, $n$ denotes an extension to a
neighborhood of $\domeg$ of the unit outward normal to $\domeg$. But such penalization terms do not
tend to $0$ if we choose as above $x=y$. Moreover, it is known that a cone condition should hold for the
subsolution. So, here we require by $(iii)$ that for any $\xb\in \domeg$, there exist $(\xe,\ye)\to
(\xb,\yb)$ such that $$ u(\xe)-v(\ye)-\left\vert \frac{\xe-\ye}{\eps}+n(\ye)\right\vert^2\to
u(\xb)-v(\xb)\;.$$ This assumption is satisfied by $\xe=\xb-\e n(\xb)$, $\ye=\xb$ if $u$ is 
continuous or if the cone condition holds for $u$.

\subsection{Penalization at infinity}\label{subsec:pen.inf}

\index{Penalization!the non-compact case}
The following result is connected to our localization procedure.
\begin{proposition}\emph{--- Penalization at infinity.}\smsp 
    Let $w:\R^N\to \R$ a bounded \usc function and $(w_\alpha)_{\alpha>0}$ a sequence of \usc
    functions such that
    \begin{enumerate}
        \item[$(i)$] $w_\alpha(x)\to -\infty$ as $|x|\to +\infty$,
        \item[$(ii)$] $w_\alpha(x)\to w(x)$ when $\alpha \to 0$ for any $x\in \R^N$.
    \end{enumerate}
Then, if $M_\alpha :=\max_{\R^N}(w_\alpha)$ and $M :=\sup_{\R^N}(w)$, we have
    $$\liminf M_\alpha \geq M\;.
    $$ Moreover, if $w_\alpha (x)=w(x)-\alpha \chi(x)$ where $\chi:\R^N\to \R$ is a coercive, locally
    bounded, \lsc function and if $x_\alpha$ is such that $w_\alpha (x_\alpha)=M_\alpha$ then
    $w(x_\alpha) \to M$ and $\alpha \chi(x_\alpha) \to 0$.
\end{proposition}

\begin{proof} By definition of the supremum, there exists a sequence $(x_k)_k$ of points in $\R^N$ such that $w(x_k)\to M$ and,
for any $k$, 
$$ w_\alpha(x_k) \leq M_\alpha\; .$$
Taking the liminf  as $\alpha$ tend to $0$ and letting $k$ tend to infinity, we obtain the first part of the result.

For the second part, we use the fact that $\chi$ is bounded from below and therefore $M_\alpha \leq M-\alpha m$, where $m=\min_{\R^N} (\chi)$. Hence $\limsup M_\alpha \leq M$ and therefore $M_\alpha \to M$. In other words
$$  w_\alpha (x_\alpha)=w(x_\alpha)-\alpha \chi(x_\alpha)\to M\; .$$
But $ -\alpha \chi(x_\alpha)\leq - \alpha m$ and therefore
$$w(x_\alpha)=M_\alpha + \alpha \chi(x_\alpha)\geq M_\alpha +\alpha m\; .$$
Hence $\liminf w(x_\alpha) \geq M$ but obviously $\limsup w(x_\alpha)\leq M$. This yields
$\lim w(x_\alpha)= M$ and, as a consequence,
$ -\alpha \chi(x_\alpha)=M_\alpha - w(x_\alpha)\to 0$.
\end{proof}


\part{Deterministic Control Problems and Hamilton-Jacobi Equations for Codimension One Discontinuities}
\label{part:codim1}
\fancyhead[CO]{HJ-Equations with Discontinuities: Codimension-$\mathbf{1}$ Discontinuities}


\chapter{Introduction : Ishii Solutions for the Hyperplane Case}
\label{chap:Ishii}

\abstract{This introduction describes the difficulties to address the simplest problems involving
discontinuities, \ie the case of a codimension~$1$ discontinuity on an hyperplane, both from the pde
and control points-of-view. The uniqueness/comparison questions are especially emphasized.} 

In this part, we consider one of the simplest and emblematic case of discontinuity for an equation
or a control problem: the case when this discontinuity is an hyperplane, say $\H=\{x_N=0\}$. In terms
of stratification, as introduced in Section~\ref{sect:whitney}, this is one of the simplest examples
of stratification of $\R^N\times(0,\Tf)$ for which $\Man{N+1}=(\Omega_1\cup\Omega_2)\times(0,\Tf)$,
$\Man{N}=\H\times(0,\Tf)$ and
$\Man{k}=\emptyset$ for any $k=0..(N-1)$, where
$$\ \Omega_1=\{x_N>0\}\;,\ \Omega_2=\{x_N<0\} \; .$$ 
For simplicity of notations, we also write $\Omega_0=\H$ and we take the convention to denote by
$e_N=(0,\dots,0,1)$ the unit vector pointing inside $\Omega_1$, so that $e_N$ is also the outward
unit normal to $\Omega_2$, see figure~\ref{fig:codim1} below.
\begin{figure}[!htp]
    \begin{center}
   \includegraphics[width=0.5\textwidth]{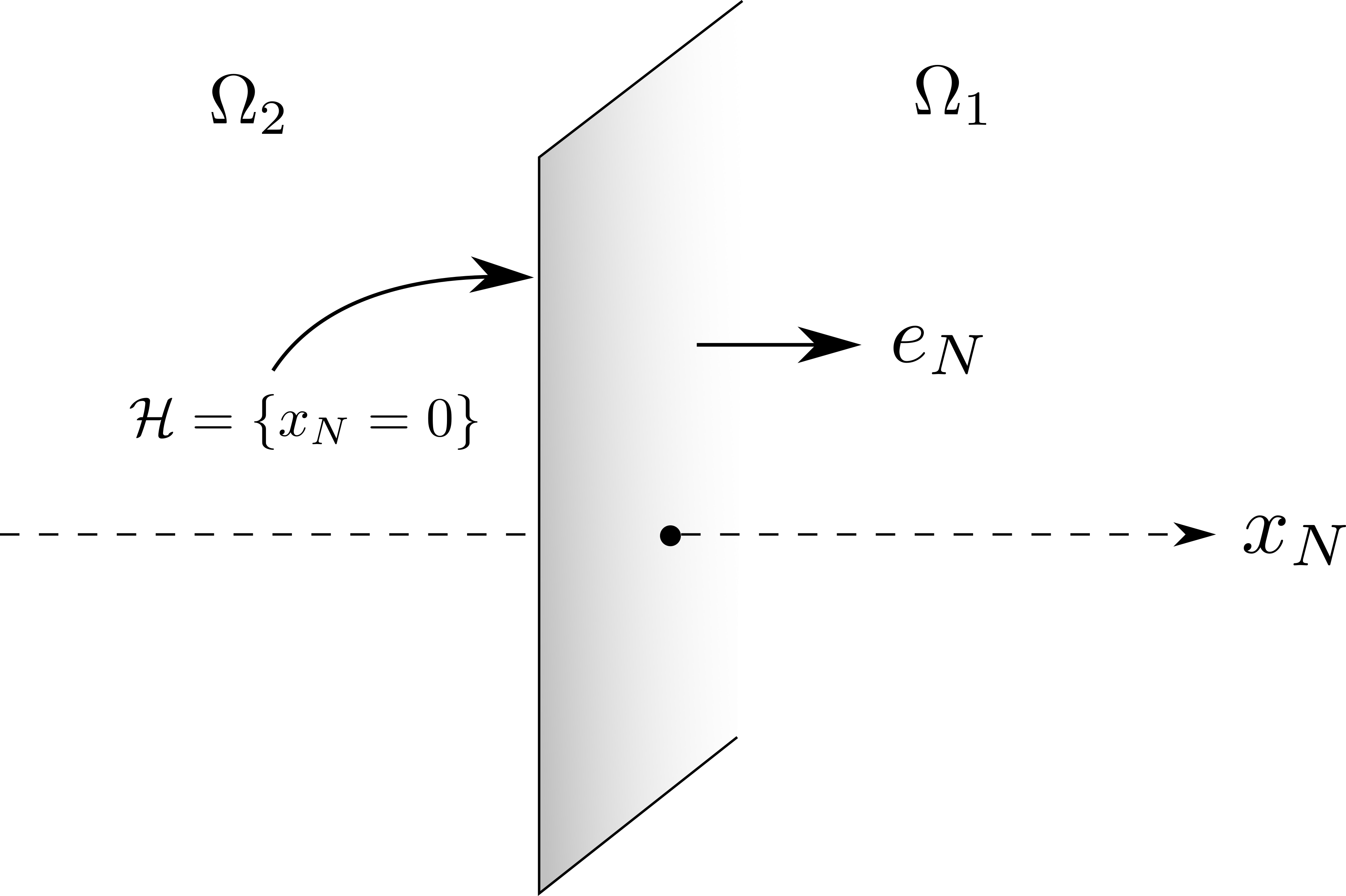}
   \caption{Setting of the codimension one case}
    \label{fig:codim1} 
   \end{center}
\end{figure}

Two types of questions can be addressed whether we choose the pde or control point of view and, in
this part, both will be very connected since we mainly consider Hamilton-Jacobi-Bellman type
equations. 

\section{The pde viewpoint}

From the pde viewpoint, the main question concerns the existence and uniqueness of solutions to the problem
\begin{equation}\label{pb:half-space}
	\begin{cases}
        u_t+H_1(x,t,u,Du)=0 & \text{ for }x\in\Omega_1\times(0,\Tf)\;,\\
        u_t+H_2(x,t,u,Du)=0 & \text{ for }x\in\Omega_2\times(0,\Tf)\;,\\	
    u(x,0)=\u0 (x) & \text{ for }x\in\R^N\;,
\end{cases}
\end{equation}
under some standard assumptions on $H_1,H_2$ and $\u0$. It is also very natural to consider a
specific control problem or pde on $\H$, which amounts to adding an equation 
\begin{equation}\label{eq:H0}
	u_t+H_0(x,t,u,D_T u)=0  \text{ for }x\in\H\;,\,
\end{equation}
where $D_Tu$ stands for the \emph{tangential derivative} of $u$, $i.e.$ the $(N-1)$ first
components of the gradient, leaving out the normal derivative.  However, for reasons that will be
exposed later in Section~\ref{sec:H0.case}, adding such a condition is not completely tractable in
the context of Ishii solutions and is more relevant in the context of flux-limited solutions or
junction conditions (see Part~\ref{part:NA}). Therefore, except for Section~\ref{sec:H0.case}, we
restrict ourselves to problem~\eqref{pb:half-space}.

As we explained in Section~\ref{sect:stab}, the conditions on $\H$ for those equations have to be
understood in the relaxed (Ishii) sense, namely
\begin{equation}\label{eq:ishii.cond}
\!\!\begin{cases}
	\max\Big(u_t+H_1(x,t,u,Du),u_t+H_2(x,t,u,Du)\Big)\geq0 \;,\\
	\min\Big(u_t+H_1(x,t,u,Du),u_t+H_2(x,t,u,Du)\Big)\leq0 \;,\\
	\end{cases}
\end{equation}
meaning that for the supersolution \resp{subsolution} condition, at least one of the inequation for
$H_1$ or $H_2$ has to hold.

\section{The control viewpoint}

From the control viewpoint, we are in the situation where different dynamics, discount factors and
costs are defined on $\Omega_1$ and $\Omega_2$. A double question arises: $(i)$ how to define a global
control problem in $\R^N$\,? $(ii)$ once this is done, if each Hamiltonian in \eqref{pb:half-space}
is associated to the control problem in the corresponding domain, is the ``usual'' value function
still the unique solution of \eqref{pb:half-space}?

In this chapter, we combine several tools introduced in Part~\ref{Part:prelim} in order to
address these problems. Notice that the present stratification \index{Stratification!in the
two-half-spaces case} of $\R^N$ is obviously a typical \AFS. So, assuming moreover that each
Hamiltonian satisfies \NCe, \TC and \Mong, we are in what we called a ``good'' framework
for treating discontinuities in the sense of Definition~\ref{GF-HJD} (here, no diffeomorphism is
needed since the stratification is flat).

\section{The uniqueness question}

As we will see, Ishii's notion of solution is not strong enough to ensure comparison (and
uniqueness) in this setting in general: this is already true for Equation~\eqref{pb:half-space} but
the situation is even worse when adding \eqref{eq:H0} on $\H$.  Let us give a brief overview of this
story here.

The general formulation of control problems described in Chapter~\ref{chap:control.tools} provides a
``natural'' control solution of \eqref{pb:half-space}, obtained by minimizing a cost over all the
possible trajectories. We denoted this solutin by $\VFm$. By Corollary~\ref{VFm-minsup}, $\VFm$ is
in fact the minimal supersolution (and solution) of \eqref{pb:half-space}. 

But we introduce another value function denoted by $\VFp$ where we minimize over a subset of those
trajectories, that are called \emph{regular}. We will show that $\VFp$ is also an Ishii
solution of \eqref{pb:half-space}, and it is even the maximal Ishii (sub)solution of
\eqref{pb:half-space}. In general $\VFm\neq\VFp$ and we provide an explicit example of such a
configuration. Finally both $\VFm$ and $\VFp$ can be characterized by means of an additional
``tangential'' Hamiltonian on $\H$.  Later in this part, we will also see that $\VFp$ is the limit
of the vanishing viscosity method.

At this point, the reader may think that there is no difference when adding \eqref{eq:H0} to problem
\eqref{pb:half-space}, after modifying in a suitable way the specific control problem on $\H$. It
is, of course, the case for $\VFm$ where again the general results of
Chapter~\ref{chap:control.tools} apply. 

But the determination of the maximal Ishii (sub)solution is more tricky: to understand why, we refer
the reader to the Dirichlet/exit time problem for deterministic control problem in a domain; it is
shown in \cite{BP2} that, if the minimal solution of the Dirichlet problem is actually given by an
analogue of the value function $\VFm$ for such problems, the maximal one is obtained by considering
the ``worse stopping time'' on the boundary (see also \cite{Ba}).  This \emph{differential game}
feature arises here in a more complicated way and we give some elements to understand it in
Section~\ref{sec:H0.case}.

In the next four chapters, we give a complete study of \eqref{pb:half-space}: we first introduce
the control problem, define and characterize $\VFm$. Then we construct and study $\VFp$. Some
uniqueness and non-uniqueness results are proved and we discuss the problem of adding \eqref{eq:H0}
in the last Chapter~\ref{sec:H0.case}.

\chapter{The Control Problem and the ``Natural'' Value Function}
\label{sect:codimIa}

\abstract{This chapter is devoted to study the properties of the
``natural'' value function $\VFm$ under the ``good assumptions'', namely the normal controllability
and the tangential continuity. The main results are that $\VFm$ can be characterized as the minimal
Ishii supersolution (and solution) of the standard HJB Equation and the unique solution of an HJB
problem provided that an additional subsolution condition is imposed on the discontinuity.}

\index{Control problem!two-half-spaces} 
Assuming that \eqref{pb:half-space} is associated to a control problem means that there exists some
triplets dynamics-discount factors-costs  $(b_i,c_i,l_i):\Omegb_i\times[0,\Tf]\times A_i\to\R^{N+3}$  
for $i=1,2$, such that, for any $(x,t,u,p)\in\Omegb_i\times(0,\Tf]\times\R\times\R^N$,
$$H_i(x,t,u,p)=\sup_{\alpha_i\in A_i}\{-b_i (x,t,\alpha_i)\cdot
p+c_i(x,t,\alpha_i)u-l_i(x,t,\alpha_i)\}\;.$$

All these $(b_i,c_i,l_i)$ can be assumed as well to be defined on $\R^N \times[0,\Tf]\times A_i$.
Moreover, in the following we assume that they satisfy the basic assumptions \HBACP and the normal
controllability assumption

\smallskip

\begin{assumption}{\NCoH}{Normal Controllability.} 
    \label{page:NCoH}
    For any $(x,t) \in \H \times [0,\Tf]$, there exists $\delta =
    \delta(x,t)$ and a neighborhood $\VV = \VV(x,t)$ such that, for any $(y,s) \in \VV$ \\[-4mm]
    $$\begin{aligned}[]
    [-\delta,\delta\,] & \subset \{b_1(y,s,\alpha_1)\cdot e_N,\ \alpha_1 \in A_1\}\quad 
    \hbox{if }(y,s)\in \Omegb_1\;,\\
    [-\delta,\delta\,] & \subset \{b_2(y,s,\alpha_2)\cdot e_N,\ \alpha_2 \in A_2\}\quad 
    \hbox{if }(y,s)\in \Omegb_2\;,
    \end{aligned}$$
    where $e_N=(0,0\cdots,0,1)\in \R^N$.
\end{assumption}

It is easy to check that Assumption \NCoH implies \NCe for $H_1$ and $H_2$ and we refer below to
assumptions \HBACP for $(b_i,c_i,l_i)$, $i=1,2$ and \NCoH as the ``standard assumptions in the
codimension-$1$ case''.

\

\section{Finding trajectories by differential inclusions}

In order to introduce the set-valued map $\BCL$, we first notice that all the equations in
\eqref{pb:half-space} have the form ``$u_t+H(x,t,u,Du)$'', which means that 
$b_i^t(x,s,\alpha_i)=-1$ for all $i=1,2$ and all $(x,s,\alpha_i)  \in\bar\Omega_i\times(0,\Tf]\times
A_i.$ Therefore, for $i=1,2$, $x\in \Omega_i$ and $t\in [0,\Tf]$ we set 
$$\BCL_i(x,t):=((b_i,-1),c_i,l_i)(x,t,A_i)$$ 
and, for $x\in \R^N, t\in (0,\Tf]$,
$$ \BCL(x,t):=
\begin{cases}
	\BCL_1(x,t) & \text{if }x\in \Omega_1\;,\\
    \BCL_2(x,t) & \text{if }x\in\Omega_2\;,\\
    \cob(\BCL_1,\BCL_2)(x,t) & \text{if }x\in\H\;,
\end{cases}
$$
where $\cob(E_1,E_2)$ denotes the closure of the convex hull of the sets $E_1,E_2$. Notice that
here, since $\BCL_1$ and $\BCL_2$ have compact images, the convex closure reduces to the union of all possible
convex combinations of elements.

For $t=0$ we need to add more information: since we consider a finite horizon problem, 
we have to be able to stop the trajectory at time $s=0$,
and we want the initial condition $u(0)=\u0 $ to be 
encoded through the Hamiltonian $H_{init}(x,u,Du)=u-\u0$. 
So, setting $\text{Init}(x):=\{(0,0),1,\u0 (x)\}$, we are led to define
\begin{equation}\label{cond:control.half-space}
\BCL(x,0):=
\begin{cases}
	\cob(\BCL_1(x,0)\cup \text{Init}(x)) & \text{if }x\in \Omega_1\;,\\
	\cob(\BCL_2(x,0)\cup \text{Init}(x)) & \text{if }x\in \Omega_2\;,\\
    \cob(\BCL_1(x,0)\cup\BCL_2(x,0)\cup\text{Init}(x)) & \text{if }x\in\H\;.
\end{cases}
\end{equation}

At this stage, we have defined rigorously $\BCL$ following the general framework described in
Part~\ref{Part:prelim}\,--\,Chapter~\ref{chap:control.tools} but, since we are mainly in a case where
$b^t=-1$, we are going to drop from now on the $b^t$-part in $\BCL$ and, in order to simplify the
notations, we just write $b=b^x$. In fact, the only place where $b^t$ plays a role is $t=0$. Indeed,
because of the convex hull, $\BCL(x,0)$ contains all the time dynamics $b^t\in[-1,0]$
However, in our case the initial conditions reduce to 
$$ u(x,0) \leq (u_0)^* (x) \quad \hbox{and}\quad v(x,0) \geq u_0(x)\quad \hbox{in  }\R^N\;,$$ 
for a subsolution $u$ and a supersolution $v$, hence they produce no additional difficulty.

The very first checking in order to solve the control problem is the
\begin{lemma}\label{lem:HBCL}
	The set-valued map $\BCL$ satisfies \HBCL.
\end{lemma}

\begin{proof}
    Concerning \HBCLa, the proof is quite straightforward by construction: first notice that since all
    the $b_i$, $l_i$, $c_i$ are bounded by some constant $M>0$, then it is the same for all the
    elements in $\BCL$. Then, by construction $\BCL(x,t)$ is closed, hence compact, and it is
    convex. It remains to see that $(x,t)\mapsto\BCL(x,t)$ is upper semi-continuous which is clear
    since each $\BCL_i(x,t)$ is upper semi-continuous and we just make a convex hull of them.

    We turn now to \HBCLb, which follows almost immediately from \eqref{cond:control.half-space}:
    $(i)$ is obviously satisfied by our choice for $b^t$ which always belongs to $[-1,0]$. Point
    $(ii)$ clearly holds if $s>0$. Indeed, if we choose $K=M$ (the constant appearing in \HBCLa),
    since $b^t=-1$ for $s>0$ we get the inequality. Now, if $s=0$ the inequality comes from the fact
    that $-Kb^t+c\geq c=1$. Point $(iii)$ is included in \eqref{cond:control.half-space} and point
    $(iv)$ follows from the fact that this condition can only happen for $s=0$ here (otherwise
    $b^t=-1$), in which case we have $\underline{c}=c=1>0$.
\end{proof}

Thanks to Theorem~\ref{thm:existence.traj} (and recalling that we have dropped the $b^t=-1$ term),
we solve the differential inclusion
\begin{equation}\label{eq:diff.half.space}
\begin{cases}
(\dot X,\dot D,\dot L)(s)\in\BCL \big(X(s),t-s\big)&\ \text{ for a.e. }s\in[0,+\infty)\;,\\[2mm]
(X,D,L)(0)=(x,0,0)\;.
\end{cases}
\end{equation}
Notice that we have used the fact that $T(s) = t-s$ when the starting point of the
$(X,T)$-trajectory is $(x,t)$. 
As we saw in Chapter~\ref{chap:control.tools}, we we mostly write 
\begin{equation}\label{notation:XDL}
    \begin{cases}
    \dot X(s) &= b\big(X(s),t-s\big)\\
    \dot D(s) &= c\big(X(s),t-s\big)\\
    \dot L(s) &= l\big(X(s),t-s\big)
    \end{cases}
\end{equation}
in order to remember that $b$, $c$ and $l$ correspond to a specific choice in the set
$\BCL(X(s),t-s)$, but when needed we will also introduce a control $\alpha(\cdot)$ to represent
$(b,c,l)$ as $(b,c,l)(X(s),t-s),\alpha(s))\;.$

Now the aim is to give a more precise description of each trajectory. For the sake of clarity, we
denote by $(b_\H,c_\H,l_\H)$ the $(b,c,l)$ when $X(s)\in\H$ which are of course obtained through a
convex combination of all the $(b_i,c_i,l_i)$, $i=1,2$. 
So, in order to take this into account, 
we introduce the ``extended control space''
$$A:=A_1\times A_2\times\tilde\Delta\quad\text{where}\quad
    \tilde \Delta:=\{(\mu_1,\mu_2)\in[0,1]^2:\mu_1+\mu_2=1\}\;,
$$
and $\mA:=L^\infty(0,\Tf;A)$. The extended control takes the form $a=(\alpha_1,\alpha_2,\mu_1,\mu_2)$
and if $x\in\H$,
$$(b_\H,c_\H,l_\H)=\mu_1(b_1,c_1,l_1)+\mu_2(b_2,c_2,l_2)\, ,$$
with $\mu_1+\mu_2=1$, where $b_1,c_1,l_1$ are computed at the point $(x,t,\alpha_1)$ and
$b_2,c_2,l_2$ at the point $(x,t,\alpha_2)$.

\begin{lemma}\label{lem:struc.traj}
	For any trajectory $(X,D,L)$ of \eqref{eq:diff.half.space} there exists a control 
	$a(\cdot)=(\alpha_1,\alpha_2,\mu_1,\mu_2)(\cdot)\in\mA$ such that
	$$\begin{aligned}(\dot X,\dot D,\dot L)(s) &= 
        (b_1,c_1,l_1)(X(s),t-s,\alpha_1(s))\ind{X(s)\in\Omega_1}\\
	    &+ (b_2,c_2,l_2)(X(s),t-s,\alpha_2(s))\ind{X(s)\in\Omega_2}\\ &+ (b_\H,c_\H,l_\H)(X(s),t-s,
        a(s))\ind{X(s)\in\H}\\
	\end{aligned}$$
	and $b_\H(X(s),t-s,a(s))\cdot e_N=0$ for almost any $s\in(t,\Tf)$ such that $X(s)\in\H$.
\end{lemma}
\begin{proof}
    Given a trajectory, we apply Filippov's Lemma (cf. \cite[Theorem 8.2.10]{AF}).  To do so, we
    define the map $g:\R^+ \times A \rightarrow \R^N$ as follows
    $$
    g(s,\a):=\begin{cases}
    b_1\big(X(s),t-s,\alpha_1\big)  &   \mbox{  if }  X(s) >0    \\
    b_2\big(X(s),t-s,\alpha_2\big) &    \mbox{  if }  X(s)< 0   \\
    b_\H\big(X(s),t-s,a \big)   &        \mbox{  if }  X(s)=0\;,
    \end{cases}
    $$
    where $a=(\alpha_1,\alpha_2,\mu_1,\mu_2)\in A$.

    We claim that $g$ is a Caratheodory map. Indeed, it is first clear that, for fixed $s$, the function
    $a\mapsto g(s,a)$ is continuous. Then, in order to check that $g$ is measurable with respect to
    its first argument we fix $a\in A$, an open set $\mathcal{O}\subset\R^N$ and evaluate
    $$ g^{-1}_a(\mathcal{O})=\big\{ s>0: g(s,a)\cap\mathcal{O}\neq\emptyset\big\} $$
    that we split into three components, the first one being
    $$ g^{-1}_a(\mathcal{O})\cap \{s>0:X(s) < 0\}= 
    \big\{ s>0: b_1(X(s),t-s,\alpha_1)\in\mathcal{O}\big\}\cap \{s>0:X(s)< 0\}\;. $$
    Since the function $s\mapsto b_1(X(s),t-s,\alpha_1)$ is continuous, this set is the intersection
    of open sets, hence it is open and therefore measurable. The same argument works for the other
    components, namely $\{s>0:X(s)<0\}$ and $\{s>0: X(s)= 0\}$ which finishes the claim.

    The function $s\mapsto \dot X(s)$ is measurable and, for any $s$, the differential inclusion
    implies that 
    $$\dot X(s) \in g(s,A)\; ,$$
    therefore, by Filippov's Lemma, there exists a measurable map
    $a(\cdot)=(\alpha_1,\alpha_2,\mu_1,\mu_2)(\cdot) \in \mA$ such that \eqref{eq:g}  is fulfilled.
    In particular, by the definition of $g$, we have for a.e. $s \in [0,\Tf]$
    \begin{equation}  \label{eq:g}
    \dot X(s)= \begin{cases}
    b_1\big(X(s),t-s,\alpha_1(s)\big)  &   \mbox{  if }  X(s) >0    \\
    b_2\big(X(s),t-s,\alpha_2(s)\big) &    \mbox{  if }  X(s)< 0   \\
    b_\H\big(X(s),t-s,\a(s)\big)   &  \mbox{  if }  X(s)=0.
    \end{cases}
    \end{equation}

    The last property is a consequence of Stampacchia's theorem (see for instance \cite{GT}):
    setting $y(s):=X_N(s)$, then $\dot y(s)=0$ almost everywhere on the set $\{y(s)=0\}$. But $\dot
    y(s)=b_\H(X(s),t-s,a(s))\cdot e_N$ on this set, so the conclusion follows.
\end{proof}

\section{The $\VFm$ value function}

Solving \eqref{eq:diff.half.space} with $\BCL$ yields a set $\mT(x,t)$ of all admissible
trajectories, without specific condition on $\H$ for \eqref{pb:half-space} (see
Section~\ref{sec:VF}). Changing slightly the notations of this section to emphasize the role of the
control $a(\cdot)$, we first define the value function
$$\VFm(x,t):=\inf_{\mT(x,t)}\left\{\int_0^tl(X(s),t-s,a(s))\exp(-D(s))\ds+\u0
(X(t))\exp(-D(t))\right\}\;,$$
and the aim is now to prove that $\VFm$ is a viscosity solution of
\eqref{pb:half-space}.\index{Viscosity solutions!value functions as} To do so, we use the control
approach described in Section~\ref{Gen-DCP}: recalling that we use the notation $b$ for $b^x$, the
``global'' Hamiltonian is given by
$$ \F (x,t,u,(p_x,p_t)):= \sup_{(b,c,l)\in\BCL (x,t)}\big(-(b,-1)\cdot (p_x,p_t)+cu-l\big)\;.$$
Writing $p$ for $p_x$ in order to simplify the notations, we decompose 
$$\F (x,t,u,(p_x,p_t))= p_t+H(x,t,u,p)\; ,$$ 
where $H(x,t,u,p)=H_i(x,t,u,p)$ if $x\in\Omega_i$ for $i=1,2$.
By the upper-semicontinuity of $\BCL$, $H$ and $\F$ are upper-semi-continuous and we have the
\begin{lemma}\label{lem:global.super}
    If $x\in\H$ then, for all $t\in [0,\Tf]$, $r\in \R$, $p_x=p\in \R^N$
	$$H(x,t,r,p)=\max\Big(H_1(x,t,r,p),H_2(x,t,u,p)\Big)\;.$$
    As a direct consequence, for any $x\in \H$, $t\in [0,\Tf]$, $u\in \R$, $p_x=p\in \R^N$, 
    $p_t\in \R$
    $$\begin{aligned}
        \F (x,t,u,(p_x,p_t)) &=  \max \big(p_t+ H_1(x,t,u,p),p_t+ H_2(x,t,u,p)\big)\;,\\
        \F_* (x,t,u,(p_x,p_t) &)=\min \big(p_t+ H_1(x,t,u,p),p_t+ H_2(x,t,u,p)\big)\;.
    \end{aligned}$$
\end{lemma}

\begin{proof}
    If $(b,c,l)\in\BCL(x,t)$, it can be written as a convex combination of some
    $(b_i,c_i,l_i)\in\BCL_i (x,t)$, $i=1,2$, and thefore the same is true for $-b\cdot p+cu-l$, namely
    $$ -b\cdot p+cr-l = \sum_i \mu_i(-b_i\cdot p+c_i r -l_i)\; ,$$
    for some $0\leq \mu_i \leq 1$ with $\sum_i \mu_i=1$. Since $(-b_i\cdot p+c_i r -l_i)\leq
    H_i(x,t,u,p)$, we deduce that $-b\cdot p+cr-l \leq \max\big(H_1(x,t,r,p),H_2(x,t,u,p)\big)$ 
    and therefore 
    $$H(x,t,r,p)\leq \max\Big(H_1(x,t,r,p),H_2(x,t,u,p)\Big) \;.$$
    But $H(x,t,r,p)\geq (-b_i\cdot p+c_i r -l_i)$ for any $(b_i,c_i,l_i)\in\BCL_i (x,t)$
    so that $H(x,t,r,p)\geq H_i(x,t,r,p)$ for $i=1,2$. The representation of $H$ as the max follows
    immediately. 

    Concerning $\F$, the first equality (as a maximum) is trivial and the representation formula for
    $\F_*$ derives directly from its definition as the $\liminf$, knowing that of course $H_1$ and
    $H_2$ are both continuous up to $\H$. 
\end{proof}

Then, by using all the results of Section~\ref{Gen-DCP}, we have the
\begin{proposition}\label{prop:ishii}\emph{--- Minimality of the value function.}\smsp
    Assume that the ``standard assumptions in the codimension-$1$ case'' are satisfied.  Then the
    value function $\VFm$ is an Ishii viscosity solutions of \eqref{pb:half-space}.  Moreover $\VFm$
    is the minimal supersolution of \eqref{pb:half-space}.
\end{proposition}

We leave the proof of the reader since it immediately follows from Theorem~\ref{SP} and
Corollary~\ref{VFm-minsup}.  This result gives a good amount of information on $\VFm$ but not all of
them. 

To go further, we have to examine more carefully the viscosity inequality on $\H$ which is done in
the next section. However, in order to do so we need first to make sure that $(\VFm)^*$ is regular
in the sense of Definition~\ref{def:regular}. We provide below a direct ``control proof'' of this fact but
for a pde proof, the reader can also check that Proposition~\ref{reg-sub} applies here since we
assume \NCoH. Notice also that the proof below only uses ``outward normal controllability'' both from
$\Omega_1$ and $\Omega_2$.
\begin{lemma}\label{limsuponH}
    Assume that the ``standard assumptions in the codimension-$1$ case'' are satisfied, then 
    $$ ((\VFm)_{|\H\times (0,\Tf)})^* = (\VFm)^*\quad \hbox{on  }\H\times (0,\Tf)  \;,$$
    where $(\VFm)_{|\H\times (0,\Tf)}$ denotes the restriction to $\H\times (0,\Tf)$ of $\VFm$.
\end{lemma}

\begin{proof} 
    Let $(x,t)\in\H\times(0,\Tf)$.
    By definition of $(\VFm)^*$,  there exists a sequence $(x_n,t_n)\to(x,t)$ such that
    $\VFm(x_n,t_n)\to (\VFm)^* (x,t)$. The statement of Lemma~\ref{limsuponH} means that we can
    assume that $x_n \in \H$. Indeed, if $x_n \in \Omega_1$, we use the normal controllability
    assumption \NCoH at $(x,t)$: there exists $\delta>0$ and a control $\alpha_1$ such that
    $b_1(x,t,\alpha_1) \cdot e_N=-\delta<0$.  Considering the trajectory with constant control
    $\alpha_1$ 
    \begin{equation}\label{Yn}
        \dot Y(s) = b_1(Y(s),t_n-s,\alpha_1)\quad,\quad Y(0)=x_n ,
    \end{equation}
    it is easy to show that $\tau^1_n$, the first exit time of the trajectory $Y$ from $\Omega_1$
    tends to $0$ as $n\to +\infty$.  By the Dynamic Programming Principle, denoting $(\tilde
    x_n,\tilde t_n)= (X(\tau^1_n),t-\tau^1_n)$, we have
    $$ 
        \VFm (x_n,t_n)\leq \int_{0}^{\tau^1_n}l\big(Y(s),t_n- s,\alpha_1\big)\,e^{-D(s)}\ds +
    	\VFm (\tilde x_n,\tilde t_n)\,e^{-D(\tau^1_n)}= \VFm (\tilde x_n,\tilde t_n) + o_n(1)\;,
    $$
    where $o_n(1)\to 0$. Therefore $(\tilde x_n,\tilde t_n) \to (x,t)$, $\VFm(\tilde x_n,\tilde t_n)
    \to (\VFm)^* (x,t)$ and $\tilde x_n \in \H$, which is exactly what we wanted to prove. 
    The same results holds if $x_n\in\Omega_2$ using a control such that 
    $b_2(x,t,\alpha_2) \cdot e_N=\delta>0$.
\end{proof}

\section{The complementary equation}\label{subsec:complemented}

This section is motivated in particular by Lemma~\ref{lem:struc.traj} where the term
$(b_\H,c_\H,l_\H)$ plays a key role as a coupling between the control problems in $\Omega_1$ and
$\Omega_2$. 

Following Section~\ref{sect:Local.Comparison}, we introduce the tangential elements in $\BCL$ which
maintain the trajectories on $\H$: for any $x\in\H$, $t\in [0,\Tf]$, we set
$$\BCL_T(x,t):=\big\{(b,c,l)\in\BCL(x,t):b\cdot e_N=0\big\}\;.$$
Similarly we define $\B_T(x,t)$ for the set-valued map of tangential dynamics: any
$b\in\B_T(x,t)$ can be expressed as a convex combination
\begin{equation}\label{eq:convex.comb.b}
b=\mu_1 b_1+\mu_2 b_2
\end{equation}
for which $(\mu_1 b_1+\mu_2 b_2)\cdot e_N=0$ with $\mu_1+\mu_2=1$, $\mu_1,\mu_2\in [0,1]$.
We also introduce tangential Hamiltonian which was already considered
\begin{equation}\label{def:HT}
\HT(x,t,u,p):=\sup_{\BCL_T(x,t)}\big\{-b\cdot p+cu-l\big\}\;.
\end{equation}
Notice that $p_t+ \HT(x,t,u,p)=\F^N(x,t,u,(p,p_t))$ on $\Man{N}=\H\times(0,\Tf)$ and,
by Lemma~\ref{tgfields} with $k=N$, the Hamiltonian $\HT$ satisfies \TC; in particular, $\\HT$ is continuous in
$x,t$, uniformly with respect to $(u,p)$ in compact sets. Such property can also be obtained by using the
representation formula given by Lemma~\ref{lem:H1m.H2p.a}.

Before deriving an $\HT$-subsolution property, we need first the following preliminary result which
allows us to build trajectories which remains on $\H$, at least for some time.
\begin{lemma}\label{bon-bcl}
    Let $(x,t) \in \H \times (0,\Tf)$ and $(b,c,l)\in \BCL_T(x,t)$, obtained as a convex
    combination $(b,c,l)=\mu_1 (b_1,c_1,l_1)+\mu_2(b_2,c_2,l_2)$. 
    If $$(b_1(x,t,\alpha_1)\cdot e_N)\,\cdot\,(b_2(x,t,\alpha_1)\cdot e_N)<0\;,$$
    there exists a neighborhood $\VV$ of
    $(x,t)$ in $\H \times (0,\Tf)$ and a Lipschitz continuous map $\psi : \VV \to \R^N\times \R \times
    \R$, such that $\psi(x,t)=(b,c,l)$ and  $\psi(y,s) = (\tilde b(y,s), \tilde c(y,s), \tilde l(y,s))
    \in \BCL_T(y,s)$ for any $(y,s) \in \VV$.  
\end{lemma}

\begin{proof}
    Our assumption means that
    $$(\mu_1 b_1(x,t,\alpha_1) +\mu_2 b_2(x,t,\alpha_2))\cdot e_N=0\;.$$
    Now, if $(y,s)$ is close enough to $(x,t)$ we set
    \begin{equation*}
        \mu^\sharp_1(y,s):= \frac{b_2(y,s,\alpha_2) \cdot
	    e_N}{(b_2(y,s,\alpha_1)-b_1(y,s,\alpha_1))\cdot e_N}\;,\quad \mu^\sharp_2:=1-\mu^\sharp_1\;.
    \end{equation*}
    By this choice we have $0\leq\mu^\sharp_1,\mu^\sharp_2\leq 1$ and $\left(
    \mu^\sharp_1(y,s)b_1(y,s,\alpha_1)+  \mu^\sharp_2(y,s)b_2(y,s,\alpha_2)\right)\cdot e_N=0$,
    which yields a tangential dynamic which is well-defined as long as
    $(b_2(y,s,\alpha_1)-b_1(y,s,\alpha_1))\cdot e_N\neq 0$. In particular this is true in a
    neighborhood of $(x,t)$.

    Then the function $\psi$ given by
    $$ \psi(y,s):= \mu^\sharp_1(y,s)(b_1,c_1,l_1) + \mu^\sharp_2(y,s)(b_2,c_2,l_2)\; ,$$
    satisfies all the desired properties: it is Lipschitz continuous since $b_1, b_2$ are Lipschitz
    continuous in $x,t$ and since $\mu^\sharp_1(x,t)=\mu_1$, $\mu^\sharp_2(x,t)=\mu_2$, 
    $\psi(x,t)=(b,c,l)$.
\end{proof}

We now prove that a complementary subsolution inequality holds on $\H$:
\begin{proposition}\label{prop:complemented.one-d}
    Assume that the ``standard assumptions in the codimension-$1$ case'' are satisfied.
    Then the value function $\VFm$ satisfies the viscosity inequality
	$$(\VFm)^*_t+\HT\Big(x,t,(\VFm)^*,D_T(\VFm)^*\Big)\leq 0\quad \hbox{on  }\H\times (0,\Tf)\;.$$
\end{proposition}

We point out that in Proposition~\ref{prop:complemented.one-d}, the 
$\H\times (0,\Tf)$-viscosity inequality means that we look at maximum points of
$(\VFm)^*-\phi$ on $\H\times (0,\Tf)$ where $\phi$ is a smooth test-function on
$\H\times (0,\Tf)$.

\begin{remark}
    In other words, $\VFm$ is an Ishii solution satisfying a complemented $\HT$-inequality on $\H$.
    As we will see in Part~\ref{stratRN}, this can be interpreted as $\VFm$ being a stratified
    solution of the problem. We will actually prove that it is the unique stratified solution.
\end{remark}

\begin{proof}
    If $\phi$ is a smooth test-function on $\H\times(0,\Tf)$, we have to prove that, if $(x,t) \in
    \H\times  (0,\Tf)$ is a maximum point on $\H\times (0,\Tf)$ of $(\VFm)^*-\phi$, then (assuming
    without loss of generality that $(\VFm)^*(x,t)=\phi(x,t)$),
    $$\phi_t (x,t) +\HT (x,t,\phi (x,t) ,D_T \phi(x,t))\leq0\quad \hbox{on  }\H\times (0,\Tf)  \;.$$

    \noindent\textbf{(a)} \emph{Using the dynamic programming principle ---}
    By Lemma~\ref{limsuponH}, we can pick a sequence $(x_n,t_n)\to(x,t)$ such that $\VFm(x_n,t_n)\to
    (\VFm)^* (x,t)$ with $x_n \in \H$ for all $n\in\N$. By the dynamic programming principle, for
    any $\tau>0$ and any trajectory $(X_n,a_n)$ in $\mT(x_n,t_n)$ we have
    \begin{equation}\label{dyn.prog.sub.ishii0}
        \VFm(x_n,t_n)\leq \int_{0}^{\tau}
        l\big(X_n(s),t_n- s,a_n(s)\big)\,e^{-D_n(s)}\ds + \VFm ( X_n(\tau), t_n-\tau)
        \,e^{-D_n(\tau)}\;.
    \end{equation}
    Our aim is to show that this inequality implies
    $$ \phi_t(x,t) -b\cdot D\phi(x,t) +c \phi(x,t) - l \leq 0\; ,$$
    for any $(b,c,l)\in\BCL_T(x,t)$, which will give the conclusion $\HT\leq0$.
    However, replacing $\VFm$ by $\phi$ above can be done only for trajectories which stay on $\H$,
    at least for some interval $[0,\tau]$. 

    \smallskip

    \noindent\textbf{(b)} \emph{Constructing a trajectory which stays on $\H$ ---}
    We start from the fact that by definition of $\BCL_T(x,t)$, $(b,c,l)$ can be expressed as a convex
    combination of the $(b_i,c_i,l_i)$ for $i=1,2$, namely 
    $$(b,c,l) =\mu_1 (b_1,c_1,l_1) +\mu_2 (b_2,c_2,l_2)$$
    with $\mu_1+\mu_2=1$, $\mu_1,\mu_2\in [0,1]$ and $(\mu_1 b_1+\mu_2 b_2)\cdot e_N=0$. We denote
    by $\alpha_i$ the control which is associated to $(b_i,c_i,l_i)$

    Slightly modifying $b_1$ and $b_2$ by using the normal controllability on $\H$, we may assume
    without loss of generality that $b_1\cdot e_N \neq 0$ and  $b_2\cdot e_N \neq 0$ while keeping
    $(\mu_1 b_1+\mu_2 b_2)\cdot e_N=0$. Therefore, either $b_1 \cdot e_N<0<b_2 \cdot e_N$ or 
    $b_1 \cdot e_N>0>b_2 \cdot e_N$ but in both cases Lemma~\ref{bon-bcl} provides us with a
    function $\psi$ that we use to solve the ode
    $$(\dot{X}_n(s), \dot{D}_n(s),\dot{L}_n(s)) = \psi(X_n(s),t_n-s)\; ,$$
    with $(X_n(0),D_n(0),L_n(0))=(x_n,0,0)$.

    Because of the properties of $\psi$, the Cauchy-Lipschitz Theorem implies that there exists a unique
    solution which, for $(x_n,t_n)$ close enough to $(x,t)$, is defined on a small but fixed ($i.e.$
    independent of $n$) interval of time $[0,\tau]$ and $(X_n,D_n,L_n)\in \mT(x_n,t_n)$ for any $n$. 
    Moreover, $X_n\in\H$ on $[0,\tau]$.

    \smallskip

    \noindent\textbf{(c)} \emph{Deriving the tangential inequality ---}
    Since $\VFm(x_n,t_n)=(\VFm)^*(x,t)+o_n(1)=\phi(x,t)+o_n(1)$ while $\VFm\leq\phi$ on
    $\H\times(0,\Tf)$, using $X_n$ in \eqref{dyn.prog.sub.ishii0} we get
    \begin{equation}\label{dyn.prog.sub.ishii1}
	    \phi(x_n,t_n)+o_n(1)\leq \int_{0}^{\tau}
        \dot L_n(s)\,e^{-D_n(s)}\ds + \phi ( X_n(\tau), t_n-\tau)
        \,e^{-D_n(\tau)}\;.
    \end{equation}

    We first let $n$ tend to infinity. Due to the Lipschitz property of $\psi$, up to extraction we
    see that $(X_n,D_n,L_n)\to(X,D,L)$ in $W^{1,\infty}$ where at least on $[0,\tau]$,
    $$(\dot{X}(s), \dot{D}(s),\dot{L}(s)) = \psi(X(s),t-s)\; ,$$
    $X(s)\in\H$ for any $s\in[0,\tau]$ and $(X(0),D(0),L(0))=(x,0,0)$.
    So, passing to the limit in \eqref{dyn.prog.sub.ishii1} yields 
    $$
	\phi(x,t)\leq \int_{0}^{\tau}\dot L(s)\,e^{-D(s)}\ds +
	\phi ( X(\tau), t-\tau)\,e^{-D(\tau)}\;.
    $$
    On the other hand, since $\phi$ is smooth on $\H\times(0,\Tf)$, the following expansion holds:
    $$\phi(X(\tau),t-\tau))e^{-D(\tau)}=\phi(x,t)+\int_0^\tau \Big(D_x\phi(\xi_s)\dot X(s)
    -\partial_t\phi(\xi_s)-\dot D(s)\phi(\xi_s)\Big)e^{-D(s)}\,\ds$$
    where $\xi_s$ stands for $(X(s),t-s)$. Combining both integrals, we arrive at 
    $$
    0\leq \int_0^\tau \Big(-\partial_t\phi(\xi_s)+\dot X(s)\cdot D\phi(\xi_s)-\dot D(s)\phi(\xi_s)
    +\dot L(s)\,e^{-D(s)}\Big)\exp(-D(s))\ds\;.$$
    Finally, after divinding by $\tau$ and sending $\tau\to0$ the conclusion follows from the fact
    that $\psi$ is continuous and $\psi(x,t)=(\dot X(0),\dot D(0),\Dot L(0))=(b,c,l)$: we get
    $$ \phi_t(x,t) -b\cdot D\phi(x,t) +c \phi(x,t) - l \leq 0\;$$
    for any $(b,c,l)\in\BCL_T(x,t)$, which implies that $\HT(x,t,\phi,D\phi)\leq0$.
\end{proof}

\section{A characterization of $\VFm$}

The previous section showed that $\VFm$ satisfies an additional subsolution inequality on
$\H\times (0,\Tf)$. The aim of this section is to prove that this additional inequality is enough to
characterize it.

The precise result is the
\begin{theorem}\label{thm:minimal.charac}\emph{--- Characterization of the minimal vale
    function.}\smsp
	Assume that the ``standard assumptions in the codimension-$1$ case'' are satisfied. Then 
	$\VFm$ is the unique Ishii solution of \eqref{pb:half-space} such that 
\begin{equation}\label{eq:HTregineqsub}
u_t+\HT(x,t,u,D_Tu)\leq0\quad\text{on}\quad\H\times (0,\Tf)\;.
\end{equation}
\end{theorem}

\begin{proof} 
    The proof is obtained by a combination of arguments which will also be used in
    Part~\ref{stratRN} for stratified problems. 

    We recall that we already know (cf. Proposition~\ref{prop:ishii}) that $\VFm$ is the minimal
    Ishii supersolution of \eqref{pb:half-space}. Therefore we only need to compare $\VFm$ with
    subsolutions $u$ such that $u_t+\HT(x,t,u,D_Tu)\leq0$ on $\H\times (0,\Tf)$, showing that
    $\VFm\geq u$ in $\R^N\times [0,\Tf]$. 

    Though the proof can be reduced to a mere list of several arguments already exposed in
    Part~\ref{Part:prelim}, we provide below more explanations and redo most of them in the simpler
    hyperplane context for the readers's convenience.

\medskip

\noindent\textbf{Step 1:} \emph{Reduction to a local comparison result \LCR} --
    As already noticed in Part~\ref{Part:prelim} (see Remarks on page~\pageref{page:c.positif}),
    setting $\tilde{u}(x,t):=\exp(Kt)u(x,t)$ for $K>0$ large enough allows to reduce the proof to
    the case where $c_i\geq0$ for any $(b_i,c_i,l_i)\in\BCL_i(x,t)$, $i=1,2$. As a consequence, we
    can assume that the $H_i$ ($i=1,2$) are nondecreasing in the $u$-variable, and that $\HT$ enjoys
    the same property.

    Then, rewriting here some arguments already given in Section~\ref{sect:htc} and using that the
    $c_i$ are positive, we notice that, for $\delta >0$ small enough,
    $\psi(x,t)=-\delta(1+|x|^2)^{1/2}-\delta^{-1}(1+t)$ is not only a $\delta/2$-strict subsolution
    \eqref{pb:half-space}, but also for the $\HT$-equation on $\H\times (0,\Tf)$ and we can also
    assume that $\psi\leq u$ in $\R^N \times [0,\Tf]$. For $\mu\in(0,1)$, setting
    $$u_\mu(x,t):=\mu u(x,t)+(1-\mu)\psi(x,t)$$
    yields an $\eta$-strict subsolution $u_\mu$ for some $\eta(\mu,\delta)>0$. By this, we mean that
    each inequality in \eqref{pb:half-space} is $\eta$-strict for $u_\mu$ but also that
    $(u_\mu)_t+\HT(x,t,u_\mu,D u_\mu)\leq\eta<0$ on $\H\times (0,\Tf)$. This claim is obvious for
    the initial data, let us prove it for instance for $H_1$.

    Using the convexity property of $H_1$ in $r,p$,	we get successively
	$$\begin{aligned}
	& (u_\mu)_t + H_1(x,t,u_\mu,D u_\mu)\\
	&= \mu u_t + (1-\mu) \psi_t +H_1(x,t,\mu u+(1-\mu)\psi,\mu Du+(1-\mu) D\psi)\\
	&\leq \mu u_t + (1-\mu) \psi_t  + \mu H_1(x,t,u,Du)+(1-\mu)H_1(x,t,\psi,D\psi)\\
	&\leq \mu \big\{u_t +  H_1(x,t,u,Du)\big\}+(1-\mu)\big\{\psi_t  + H_1(x,t,\psi,D\psi)\big\}\\
	&\leq \mu\big\{u_t+H_1(x,t,u,Du)\big\}-(1-\mu)(\delta/2)\leq -(1-\mu)(\delta/2)<0\;.
	\end{aligned}$$
    The same is valid for $H_2$ and $\HT$ for similar reasons.
    Moreover, by construction $u_\mu-\VFm\to -\infty$ as $|x|\to+ \infty$ since $\psi(x,t) \to
    -\infty$ as $|x|\to+ \infty$, so that \LOCa is satisfied for any of those Hamiltonians. 

    Checking \LOCb is easier: if we are looking for a comparison result around the point
    $(x_0,t_0)$, it is enough to use 
    $$u_{\delta'}(x,t):=u(x,t)-\delta'(|x-x_0|^2+|t-t_0|^2)$$
    for $\delta'>0$ small enough. Thus we are in the situation where a \LCR is enough to ensure a
    \GCR.

    In order to prove that \LCR holds, we introduce $\cyl$, a (small) cylinder around $(x,t)$ where
    we want to perform the \LCR. Notice that of course, if $x\in\Omega_1$ or $\Omega_2$, then taking
    $r$ small enough reduces the proof to the standard comparison result since in this case, $\cyl$
    does not intersect with $\H$. Thus, we assume in the following that $x\in\H$. Our aim is to
    use Lemma~\ref{lem:comp.fundamental} with $\mathcal{M}:=(\H\times[0,\Tf])\cap\overline{\cyl}$
    and $\mathbb{F}^\mathcal{M}(x,t,r,(p_x,p_t)):=p_t + \HT(x,t,r,p_x)$.

\medskip

\noindent\textbf{Step 2:} \emph{Approximation of the subsolution} --
    We wish to use an approximation by convolutions (inf-convolution and usual convolution with a
    smoothing kernel) for the subsolution as in Proposition~\ref{C1-reg-by-sc}; to do so, we
    introduce a slightly larger cylinder $\cylp$ where $r'>r$ and $h'>h$ are fixed in order to have
    some ``room'' for those convolutions.  From Step 1, we know that $u_\mu$ is an $\eta$-strict
    subsolution of \eqref{pb:half-space} in $\cylp$ for some $\eta=\eta(\mu,\delta)$. 

    Since \HConv, \NCe, \TC and \Monu are satisfied for all the Hamiltonians, we deduce from
    Proposition~\ref{C1-reg-by-sc} that there exists a sequence $(u_{\mu,\eps})_\eps$ of
    $C^0(\cylb)\cap C^1(\mathcal{M})$ functions which are all $(\eta/2)$-strict subsolutions of
    \eqref{pb:half-space} in some smaller cylinder $Q(\eps)\subset\cylp$, and $Q(\eps)\to\cylp$ as
    $\eps\to0$ in the sense of the euclidian distance in $\R^{N+1}$.  Hence, for $\eps$ small
    enough, we can assume with no restriction that $\cyl\subset Q(\eps)\subset\cylp$ so that
    $u_{\mu,\eps}$ is an $(\eta/2)$-strict subsolution in $\cyl$. 

    This has two consequences:
    \begin{enumerate}
    \item[$(a)$] for any $\eps>0$ small enough, $(u_{\mu,\eps})_t+\HT(x,t,u_{\mu,\eps},D_T
    u_{\mu,\eps})\leq -\eta/2<0$ in $\mathcal{M}$ and in a classical sense since $u_{\mu,\eps}$ is
    $C^1$ on $\mathcal{M}$; 
    \item[$(b)$] since $u_{\mu,\eps}$ is an $(\eta/2)$-strict subsolution in
    $\mathcal{O}:=\cyl\setminus\mathcal{M}$ (for the Hamiltonians $H_1,H_2$) and a \LCR holds there,
    we use the subdynamic programming principle for subsolutions (cf.
    Theorem~\ref{thm:sub.Qk.dpp.bt}) which implies that each $u_{\mu,\eps}$ 
    satisfies an $(\eta/2)$-strict dynamic programming principle in $\cyl[\mathcal{M}^c]$.
    \end{enumerate}
    These two properties allow us to make a \LCR in $\cyl$ in the final step.

\medskip

\noindent\textbf{Step 3:} \emph{Performing the local comparison} --
    From the previous step we know that for each $\eps>0$, $u=u_{\mu,\eps}$ satisfies the hypotheses
    of the ``Magical Lemma'' (Lemma~\ref{lem:comp.fundamental})\index{Magical Lemma!for $\VFm$}.  Using $v:=\VFm$ as supersolution in this lemma, we deduce
    that
	$$\forall (y,s)\in\cylb\setminus\partial_P\cyl\;,\quad 
    (u_{\mu,\eps}-\VFm)(y,s)<\max_{\cylb}(u_{\mu,\eps}-\VFm)\;.$$
    Using that $u_\mu=\limssup u_{\mu,\eps}$, this yields a local comparison result (with inequality
    in the large sense) between $u_\mu$ and $\VFm$ as $\eps\to0$.  By step 1, we deduce that the
    \GCR holds: $u_\mu\leq \VFm$ in $\R^N\times [0,\Tf]$, and sending finally $\mu\to1$ gives that
    $u\leq\VFm$.

	The conclusion is that if $u$ is an Ishii solution such that $u_t+\HT(x,t,u,D_Tu)\leq0$ on $\H$, 
	necessarily $u\equiv\VFm$, which ends the proof.
\end{proof}

\chapter{A Less Natural Value Function, Regular Dynamics}
\label{sect:rsd}

\abstract{A new value function $\VFp$ is introduced by defining ``regular trajectories''. Under
the ``good assumptions'', the main results are that $\VFp$ can be characterized as the maximal Ishii
subsolution (and solution) of the standard HJB Equation; it is also the unique solution of
an HJB problem provided an additional subsolution condition is imposed on the discontinuity.
A stability result is also obtained for ``regular trajectories''.
}

While studying $\VFm$ we introduced the set $\BCL_T$, containing the dynamics tangent to $\H$ in
order to examining the trajectories which remain on $\H$. The new point in this section is to remark
that there are two different kinds of dynamics that allow to stay on $\H$, leading to the
construction of a second value function.

\section{Introducing $\VFp$}
Let us first begin with regular trajectories:
\index{Regular strategies}
\begin{definition}\emph{--- Regular controls, dynamics, trajectories.}\smsp
    We say that $b\in\B_T(x,t)$ is \emph{regular} if $b=\mu_1 b_1+\mu_2 b_2$ while the condition
    $b_1\cdot e_N\leq0\leq b_2\cdot e_N$ holds.  We denote by 
    $$\BCL_T^\reg(x,t):=\big\{(b,c,l)\in\BCL_T(x,t): b\text{ is regular }\big\}$$ 
    the set containing the regular tangential dynamics, and $\mT^\reg(x,t)$ the set of controlled
    trajectories with regular dynamics on $\H$, $i.e.$\label{not:mTreg}
    $$\begin{aligned}
	\mT^\reg(x,t):=\Big\{ & (X,D,L) \text{ solution of $\eqref{eq:diff.half.space}$ such that }\\
	& \dot X(s)\in\B_T^\reg(X(s),t-s)\;\text{a.e. when }X(s)\in\H \Big\}\;.
	\end{aligned}
    $$
\end{definition}

In other terms, a regular dynamic corresponds to a ``push-push'' strategy: the trajectory is
maintained on $\H$ because it is pushed on $\H$ from both sides, using only dynamics coming from
$\Omega_1$ and $\Omega_2$; we may also have tangent dynamics, i.e. $b_1\cdot e_N= b_2\cdot e_N=0$.
On the contrary, the dynamic is said \emph{singular} if $b_1\cdot e_N>0$ and $b_2\cdot e_N<0$,
which is a ``pull-pull'' strategy, a quite instable situation where the trajectory remains on $\H$
because each side pulls in the opposite direction. We also recall the notations \eqref{notation:XDL}
that we use throughout this chapter.

We remark that, by \NCoH, the sets $\BCL_T(x,t)$ and $\BCL_T^\reg(x,t)$ are non-empty for any $(x,t)
\in \H$ (see Lemma~\ref{bon-bcl}). Next, for $(x,t) \in \H\times (0,\Tf)$, $r \in \R$ and
$p=(p',0)\in \R^N$, we define a second tangential Hamiltonian
\begin{equation}\label{def:HTreg}
\HTreg(x,t,r,p):=\sup_{\BCL_T^\reg(x,t)}\big\{-b\cdot p+cu-l\big\}\;,
\end{equation}
and a second value function can be defined by minimizing only on regular
trajectories:\index{Viscosity solutions!value functions as}
$$\VFp(x,t):=\inf_{\mT^\reg(x,t)}\left\{\int_0^\infty l(X(s),t-s,a(s))\exp(-D(s))\ds\right\}\;.$$
Of course it is clear that $\VFm\leq \VFp$ in $\R^N \times [0,\Tf]$ but we are going to prove more
interesting properties on $\VFp$.

The Hamiltonian $\HT^\reg$ satisfies \TC on $\H \times [0,\Tf]$; in particular, $\HT^\reg$ is continuous
with respect to $(x,t)$. Contrarily to $\\HT$, this does not follow directly from Lemma~\ref{tgfields},
but a carefull look at the proof will convince the reader that the arguments also apply to $\HT^\reg$.
As it is the case for $\HT$, an alternative proof
consists in using the representation formulas given by Lemma~\ref{lem:H1m.H2p.a}.

Proving the dynamic programming principle for $\VFp$ is done as for $\VFm$ (see Theorem~\ref{DPP}),
but using regular trajectories. So, we skip the proof of the
\begin{lemma}
    Under hypothesis \HBCL, the value function $\VFp$ satisfies
    $$
    \VFp(x,t)=
    \inf_{\cT^\mathrm{reg}(x,t)}\Big\{\int_0^\theta
    l\big(X(s),t-s,a(s)\big)\exp(-D(s)) \ds+\VFp \big(X(\theta),t-\theta)\exp(-D(\theta))\Big\}\;,
    $$
    for any $(x,t)\in\R^N\times(0,\Tf]$, $\theta >0$.
\end{lemma}

The dynamic programming principle naturally leads to a system of pde's satisfied by $\VFp$.  
But before proving this result, we want to make the following important remark: most of the
results we provided in the previous chapter for $\VFm$ were more or less direct consequences of
results given in Chapter~\ref{chap:control.tools}, in particular all the supersolution inequalities
using Lemma~\ref{lem:global.super}. However, this is not the case for $\VFp$ which requires specific
adaptations.
\begin{proposition}\label{prop:ishii-Up} 
    Assume that the ``standard assumptions in the codimension-$1$ case'' are satisfied.  Then the
    value function $\VFp$ is an Ishii solution of \eqref{pb:half-space}. Moreover $\VFp$ satisfies
    on $\H \times (0,\Tf)$ the inequality
	$$(\VFp)^*_t+\HTreg (x,t,(\VFp)^*,D_T(\VFp)^*)\leq 0\quad \hbox{on  }\H\times (0,\Tf)\;.$$
\end{proposition}

\begin{proof} 
    Of course, the only difficulties comes from the discontinuity on $\H \times (0,\Tf)$, therefore
    we concentrate on this case.

    \smallskip

    \noindent\textbf{(a)} \emph{Ishii supersolution condition in $\R^N$ ---} 
    Since a priori $\VFp$ is not continuous, we have to use semi-continuous envelopes as we did for
    $\VFm$. In order to prove that $(\VFp)_*$ is a supersolution we assume that $(x,t) \in \H
    \times (0,\Tf)$ is a strict local minimum point of $(\VFp)_*-\phi$ where $\phi$ is a smooth
    test-function in $\R^N\times(0,\Tf)$, and we can suppose w.l.o.g that $(\VFp)_*(x,t)=\phi(x,t)$.

    The first part consists in using the dynamic programming principle and follows the same lines as
    several proofs we already established so we condense a little bit some of the arguments below.
    By definition of $(\VFp)_*$, there exists a sequence $(x_n,t_n)$ which converges to $(x,t)$ such
    that $\VFp(x_n,t_n) \to (\VFp)_*(x,t)$ and by the dynamic programming principle,
    \begin{equation*}
    \VFp(x_n,t_n)= \inf_{\mT^{\rm reg}(x_n,t_n)}\Big\{
        \int_{0}^{\tau}\dot L_n(s)\,e^{-D(s)}\ds +
	\VFp \big( X_n(\tau), t_n-\tau\big)\,e^{-D(\tau)}\Big\}\;, 
    \end{equation*}
    where $\tau\ll 1$ and the $n$-index is to recall that this trajectory is associated with
    $X_n(0)=x_n$.  We use that $(i)$ $\VFp(x_n,t_n) = (\VFp)_*(x,t)+o_n(1)$ where $o_n(1) \to 0$,
    $(ii)$ $\VFp \big( X_n(\tau), t_n-\tau\big) \geq (\VFp)_* \big( X_n(\tau), t_n-\tau\big)$ and
    $(iii)$ the minimum point property, to obtain
    \begin{equation*}
    \phi (x_n,t_n)+o_n(1) \geq \inf_{\mTreg(x_n,t_n)}\Big\{
        \int_{0}^{\tau}\dot L_n(s)\,e^{-D(s)}\ds +
	\phi\big(X_n(\tau), t_n-\tau\big)\,e^{-D(\tau)}\Big\}\;.
    \end{equation*}
    Next we use  the expansion of $\phi$ along the trajectory of the differential inclusion, 
    writing $\xi_s=(X_n(s),t_n-s)$ for simplicity:
   $$
    \phi(X_n(\tau),t_n-\tau)\,e^{-D(\tau)} = \phi(x_n,t_n)
	+\int_0^\tau\Big(-\partial_t\phi(\xi_s)
	 +\dot X_n(s)\cdot D\phi(\xi_s)	 -\dot D_n(s)\phi(\xi_s)\Big)e^{-D(s)}\ds\;.
   $$
    Plugging this expansion into the dynamic programming principle and using that the global
    Hamiltonian $H$ is the sup over all the $(b,c,l)$, we are led to
    $$
	o_n(1)\leq\int_0^\tau\Big(\partial_t\phi(\xi_s)+
	H(X_n(s),t_n-s,\phi(\xi_s),D\phi(\xi_s)\Big)e^{-D(s)}\ds\;.
    $$
    Using the smoothness of $\phi$ and the upper semicontinuity of $H$ together with the facts that 
    $|X_n(s)-x| , |(t_n-s)-t|=o_n(1) + O(s)$, $e^{-D(s)}=1+O(s)$, we can replace $X_n(s)$ by $x$
    and $t_n-s$ by $t$ in the integral. Hence, for $\tau$ small enough 
    $$
        o_n(1)\leq\tau\Big(\partial_t\phi(x,t)+H(x,t,\phi(x,t),D\phi(x,t))\Big)+\tau
        o_n(1)+o(\tau)\;.
    $$
    It remains to let first $n\to\infty$, then divide by $\tau>0$ and send $\tau\to0$, which
    yields that $\partial_t\phi(x,t)+H(x,t,\phi,D\phi)\geq0$. Hence $\VFp$ satisfies the Ishii
    supersolution condition on $\H \times (0,\Tf)$.

    \medskip

    \noindent\textbf{(b)} \emph{The Ishii subsolution condition in $\R^N$ ---} 
    We have to consider $(x,t)\in\H \times (0,\Tf)$, a local maximum points of $(\VFp)^*-\phi$,
    $\phi$ being a smooth function and we assume again that $(\VFp)^*(x,t)=\phi(x,t)$.

    By definition of the upper semicontinuous envelope, there exists a sequence $(x_n,t_n)\to(x,t)$
    such that $\VFp(x_n,t_n)\to (\VFp)^* (x,t)$ and we first claim that we can assume $x_n \in \H$.
    To prove this claim, we use exactly the same argument as in the proof of Lemma~\ref{limsuponH} for
    $\VFm$ since it relies only on the normal controllability assumption \NCoH at $(x,t)$.

    Therefore, assuming that $x_n\in\H$, using the maximum point property we insert the
    test-function $\phi$ in the dynamic programming principle and get that for any regular
    control~$a(\cdot)$,
    \begin{equation}\label{dyn.prog.sub.ishii}
	\phi(x_n,t_n)+o_n(1)\leq\int_{0}^{\tau}l\big(X_n(s),t_n- s,a(s)\big)\,e^{-D(s)}\ds +
	\phi ( X_n(\tau), t_n-\tau)\,e^{-D(\tau)}\;.
    \end{equation}

    Then we argue by contradiction: if
    $$
    \min\big\{ \phi_t (x,t) +H_1\big(x,t,\phi(x,t),D\phi(x,t)\big),
    \phi_t (x,t)+H_2\big(x,t,\phi(x,t),D\phi(x,t)\big)\big\} > 0\;,
    $$
    there exists some $(\alpha_1,\alpha_2)\in A_1\times A_2$, such that, for all $i=1,2$
    \begin{equation} \label{contr.sub}
        \phi_t (x,t) -b_i(x,t,\alpha_i)\cdot D\phi(x,t)+c_i(x,t,\alpha_i)\phi(x,t)-l_i(x,t,\alpha_i)
        > 0\;, 
    \end{equation}
    and the same is true, for $n$ large enough, if we replace $(x,t)$ by $(x_n,t_n)$. Notice that,
    though the control $a(\cdot)$ in \eqref{dyn.prog.sub.ishii} is regular, this may not be the case
    a priori for $\alpha_1,\alpha_2$. Now we separate the proof in three cases according to the
    different configurations. For the sake of simplicity of notations, we just note below by $b_i$
    the quantity $b_i(x,t,\alpha_i)$.

    \noindent \textbf{Case 1 --} Either $b_1\cdot e_N>0$ or $b_2\cdot e_N<0$. In the first
    case, we use the trajectory $(X_n,D_n,L_n)$ defined by with the constant control $\alpha_1$. In
    particular
    \begin{equation}\label{Yn2}
        \dot X_n (s) = b_1(X_n (s),t_n-s,\alpha_1)\quad,\quad X_n (0)=x_n .
    \end{equation}
    Then there exists a time $\tau>0$ such that $X_n (s)\in\Omega_1$ for $s\in(0,\tau]$. Choosing
    such constant control $\alpha_1$ in \eqref{dyn.prog.sub.ishii} and arguing as above, we are led
    to
    $$
    \phi_t (x,t) -b_1(x,t,\alpha_1)\cdot D\phi(x,t)+
    c_1(x,t,\alpha_1)\phi(x,t)-l_1(x,t,\alpha_1) \leq 0\;, 
    $$
    which yields a contradiction with \eqref{contr.sub}. And the proof is the same in the second
    case, considering the trajectory associated with the constant control $\alpha_2$ in $b_2$.

    We point out that this case could have been also covered by arguments of
    Proposition~\ref{sub-up-to-b}, by extending the equation to the boundary.

    \noindent \textbf{Case 2 --} if $b_1 \cdot e_N<0<b_2 \cdot e_N$, then borrowing arguments of the
    proof of Lemma~\ref{bon-bcl}, for $(y,s)$ close enough to $(x,t)$, we can set
    \begin{equation*}
        \mu^\sharp_1(y,s):= \frac{b_2(y,s,\alpha_2) \cdot
	    e_N}{(b_2(y,s,\alpha_2)-b_1(y,s,\alpha_1))\cdot e_N}\;,
        \quad \mu^\sharp_2:=1-\mu^\sharp_1\;.
    \end{equation*}
    By this choice we have $0\leq\mu^\sharp_1,\mu^\sharp_2\leq 1$ and $\left(
    \mu^\sharp_1(y,s)b_1(y,s,\alpha_1)+  \mu^\sharp_2(y,s)b_2(y,s,\alpha_2)\right)\cdot e_N=0$,
    hence we have a regular dynamic that we use in \eqref{dyn.prog.sub.ishii}.

    We solve the ode
    \begin{equation*}
        \dot{X}^\sharp(s)=\mu^\sharp_1(X^\sharp(s),t_n-s)b_1(X^\sharp(s),t_n-s,\alpha_1)+
        \mu^\sharp_2(X^\sharp(s),t_n-s)b_2(X^\sharp(s),t_n-s,\alpha_2)\;.
    \end{equation*}
    By our hypotheses on $b_1$ and $b_2$, the right-hand side is Lipschitz
    continuous so that the Cauchy-Lipschitz theorem applies and gives a solution
    $X^\sharp(\cdot)$ which remains on $\H$, at least until some time $\tau>0$.

    Using $X^\sharp(\cdot)$ in \eqref{dyn.prog.sub.ishii} together with the associated discount and
    cost and arguing as above, we are led to
    $$\begin{aligned}
    & \mu^\sharp_1\biggl(\phi_t (x,t) -b_1(x,t,\alpha_1)\cdot D\phi(x,t)
    +c_1(x,t,\alpha_1)\phi(x,t)-l_2(x,t,\alpha_1)\biggr)\\
    + &\mu^\sharp_2\biggl(\phi_t (x,t) -b_2(x,t,\alpha_2)\cdot D\phi(x,t)
    +c_2(x,t,\alpha_2)\phi(x,t)-l_2(x,t,\alpha_2) \biggr)\leq  0\;,
    \end{aligned}$$
    a contradiction.

    \noindent \textbf{Case 3 --} The last case is when we have either $b_1 \cdot e_N=0<b_2 \cdot
    e_N$ or $b_1 \cdot e_N<0=b_2 \cdot e_N$. But using \NCoH, we can slightly modify $b_1$ or $b_2$
    by a suitable convex combination  in order to be in the framework of Case~1 or Case~2. This
    completes the proof that the Ishii subsolution condition holds on $\H\times(0,\Tf)$.

\medskip

    \noindent\textbf{(c)} \emph{The $\HT^\mathrm{reg}$-inequality ---} We do not give a specific
    proof here since this property holds for any Ishii subsolution (hence for $\VFp$ too), see
    Lemma~\ref{subsol-H}. Alternatively, this property can also be proved by similar arguments as
    for the $\HT$-inequality for $\VFm$, but using of course regular trajectories.
\end{proof}

\section{More on regular trajectories}

Let us begin by stating the stability of regular trajectories:
\begin{lemma}\label{lem:limit.reg.traj}
    Assume that all the $(b_i,c_i,l_i)$ satisfy \HBACP. For any $\eps>0$, let
    $(X,D,L)^\eps\in\mT^\reg(x,t)$ be a sequence of regular trajectories converging uniformly to
    $(X,D,L)$ on $[0,t]$. Then $(X,D,L)\in\mT^\reg(x,t)$.
\end{lemma}

Though it may seem quite natural, this result is quite difficult to obtain. It is a direct corollary of
Proposition~\ref{prop:extraction.gen} (with constant $\BCL$ and initial data) which we prove in 
Subsection~\ref{subsec:extraction} below. We recall here that since $T(s)=t-s$, we just use
trajectories in the form $(X,D,L)$ instead of $(X,T,D,L)$.

Let us focus now on the immediate consequences:
\begin{corollary}\label{cor:VFP-ssDPP}
    Assume that all the $(b_i,c_i,l_i)$ satisfy \HBACP. Then, for any $(x,t)\in\R^N\times(0,\Tf)$,
    there exists a regular trajectory $(X,D,L)\in\mT^\reg(x,t)$ such that 
    \begin{equation}\label{eq.opt.uplus}
        \VFp (x,t) =  \int_0^t l \big(X (s),t-s,a(s)\big)e^{-D(s)} \ds +
    \u0 (X (t))e^{-D(t)}\; ,
    \end{equation}
    therefore there is an optimal trajectory. Moreover, the value function $\VFp$ satisfies the
    sub-optimality principle, i.e., for any $(x,t) \in \R^N\times [0,\Tf]$ and $0<\tau< t$, we have
    $$
    (\VFp)^*(x,t) \leq  \inf_{\mT^\reg(x,t)}\left\{\int_0^\tau l \big(X (s),t-s,a(s) \big)e^{-D(s)}
    \ds + ( \VFp)^*(X (\tau),t-\tau)e^{-D(\tau)}\right\}\;,
    $$
    and the super-optimality principle, i.e.
    $$
    (\VFp)_*(x,t) \geq  \inf_{\mT^\reg(x,t)}\left\{\int_0^\tau l \big(X (s),t-s,a(s)\big)e^{-D(s)}
    \ds + ( \VFp)_*(X (\tau),t-\tau)e^{-D(\tau)}\right\}\;.
    $$
\end{corollary}

Corollary~\ref{cor:VFP-ssDPP} provides slightly different (and maybe more direct) arguments to prove
that $\VFp$ is an Ishii solution of \eqref{pb:half-space} but it relies on the extraction of regular
trajectories, which is again a rather delicate result to prove.

\begin{proof}
    We just sketch things here since everything is a straightforward application of
    Lemma~\ref{lem:limit.reg.traj}. For the existence of an optimal trajectory, we consider
    $\eps$-optimal trajectories $(X^\eps,D^\eps,L^\eps)$, $i.e.$ trajectories which satisfy 
    $$
    \VFp (x,t) \leq  \int_0^t l \big(X^\eps (s),t-s,a^\eps(s)\big)e^{-D^\eps(s)} \ds + \u0 (X^\eps
    (t))e^{-D^\eps(t)}+\eps \; .
    $$
    By applying Ascoli's Theorem on the differential inclusion, we can assume without loss of
    generality that $(X^\eps,D^\eps,L^\eps)\to(X,D,L)$ in $C([0,t])$ and $\dot L^\eps\to \dot L$
    in $L^\infty$-weak$\star$, so that for some control $a(\cdot)$, we have
    $$
    \int_0^t l\big(X^\eps (s),t-s,a^\eps(s)\big)e^{-D^\eps(s)} \ds \to 
    \int_0^t l \big(X (s),t-s,a(s)\big)e^{-D(s)} \ds.
    $$
    Then, applying Lemma~\ref{lem:limit.reg.traj} shows that
    $(X,D,L)$ is actually a regular trajectory and \eqref{eq.opt.uplus} holds.

    The proofs of the sub and super-optimality principle follow from similar arguments considering,
    for example, a sequence $(x_k,t_k)\to (x,t)$ such that $\VFp (x_k,t_k)\to  (\VFp)_*(x,t)$ and
    passing to the limit in an analogous way.
\end{proof}

\section{A Magical Lemma for $\VFp$}
\index{Magical Lemma!for $\VFp$}

Now we turn a key result in the proof that $\VFp$ is the maximal Ishii solution of \eqref{pb:half-space}.
\begin{theorem}\label{teo:condplus.VFp}\emph{--- A Magical Lemma for $\VFp$.}\smsp
    Assume that the ``standard assumptions in the codimension-$1$ case'' are satisfied.  Let $\phi\in
    C^1\big(\H\times[0,\Tf]\big)$ and suppose that $(x,t)\in\H\times(0,\Tf)$ is a
    local minimum point of $(z,s) \mapsto (\VFp)_* (z,s)-\phi (z,s)$ in
    $\H\times[0,\Tf]$.  Then the following alternative holds  \\[2mm] 
    {\bf A)} either there exist $\eta >0$, $i\in\{1,2\}$ and a control $\alpha_i (\cdot)$
	such that the associated trajectory $(X,D,L)$ satisfies $X(s) \in \Omegb_i $ with
	$\dot X(s) = b_i(X(s), t-s, \alpha_i (s))$ for all $s \in ]0,\eta]$
	and
	\begin{equation}
	    (\VFp)_*(x,t) \geq \int_0^{\eta} l_i(X(s),t-s,
	    \alpha_i(s))e^{-D(s)} \ds + (\VFp)_* (X(\eta),t-\eta)e^{-D(\eta)}\,;
	\end{equation} 
    {\bf B)} or  the following viscosity inequality holds
       	\begin{equation}\label{condhyperSUP.VFp}
	    \partial_t \phi(x,t)+\HTreg \big(x,t, (\VFp)_*(x,t), D_\H\phi(x,t)\big) \geq 0.
	\end{equation}
\end{theorem}

\begin{proof}
Using the result and the proof of Corollary~\ref{cor:VFP-ssDPP}, for any $0<\eta<t$, 
there exists a regular trajectory $X$ and a control $a$ such that
$$(\VFp)_* (x,t) \geq  \int_0^\eta l \big(X (s),t-s,a (s) \big) e^{-D(s)}\ds +
    (\VFp)_* (X (\eta),t-\eta)e^{-D(\eta)} \;.$$
Indeed, for any $\eta$ the infimum in the sub-optimality principle is achieved. 
Now there are two cases:
    \begin{enumerate}
    \item[$(i)$] Either there exists $\eta >0$ and $i\in\{1,2\}$ such that $X(s) \in \Omegb_i $ with
	$\dot X(s) = b_i(X(s), t-s, \alpha_i (s))$ for all $s \in ]0,\eta]$, from which {\bf A)} follows.
    \item[$(ii)$] Or this is not the case, which means that there exists a sequence
    $(\eta_k)_k$ converging to $0$ such that $\eta_k >0$ and $X(\eta_k) \in \H$. 
    \end{enumerate}

    In this second case, 
    $$(\VFp)_* (x,t) \geq  \int_0^{\eta_k} l \big(X (s),t-s,a (s) \big)e^{-D(s)} \ds +
    (\VFp)_* (X (\eta_k),t-\eta_k) e^{-D(\eta_k)}\;,$$
    and, assuming w.l.o.g that $\phi (x,t)=(\VFp)_* (x,t)$,  
    the minimum point property on $\H$ yields
    $$
	    \phi(x,t) \geq  \int_0^{\eta_k} l \big(X (s),t-s,a (s) \big)e^{-D(s)} \ds +
        \phi (X (\eta_k),t-\eta_k)e^{-D(\eta_k)} \;. 
    $$
    Using the notation $\xi_s=(X(s),t-s)$, we rewrite this inequality as
    $$\begin{aligned} 
        & \int_0^{\eta_k} A[\phi](s)\ds\geq0\;,\quad\text{where}\\
         A[\phi](s):= & \Big(\phi_t (\xi_s) - \dot X(s) \cdot D_x \phi(\xi_s)
        + c\big(\xi_s,a(s)\big)\phi(\xi_s) - l \big(\xi_s,a (s) \big)\Big) e^{-D(s)}
        \;.
    \end{aligned}
    $$
    In order to prove {\bf B)}, we argue by contradiction, assuming that
    \begin{equation}\label{ineq.contr.Uplus} 
	    \partial_t \phi(x,t)+\HTreg \big(x,t, (\VFp)_*(x,t), D_\H\phi(x,t)\big)<0\; ,
    \end{equation}
    and to get a contradiction we examine the sets $\E_i:=\{s\in(0,\eta_k) : X(s) \in \Omega_i\}$
    and $\E_\H:=\{s\in(0,\eta_k) : X(s) \in \H\}$.

    \smallskip

    \noindent\textbf{(a)} The case $\E_\H$ is easy: since $\dot X(s)=b_\H(X(s),t-s,a(s))$ a.e. if
    $X(s)\in\H$, by definition of $\HT^\reg$ as the supremum we get directly
    $$
        \int_0^{\eta_k} A[\phi](s)\1_{\{s\in \E_\H\}} \ds  \leq \int_0^{\eta_k} \Big\{\partial_t
        \phi(\xi_s)+\HTreg \big(\xi_s, (\VFp)_*(\xi_s), D_\H\phi(\xi_s)\big)\Big\}
        \1_{\{s\in \E_\H\}}\ds\;, 
    $$
    and this integral is stricly negative provided $\eta_k$ is small enough, thanks to
    \eqref{ineq.contr.Uplus} and the continuity of $\HT^\reg$.

    \smallskip

    \noindent\textbf{(b)} On the other hand, the sets $\E_i$ are open and therefore $\E_i=\cup_k
    (a_{i,k}, b_{i,k})$ with $a_{i,k}, b_{i,k}\in \H$.  On each interval $(a_{i,k}, b_{i,k})$,
    $\dot X(s) = b_i(X (s),t-s,\alpha_i (s) \big)$ and introducing the function $d(y)=|y_N|$, we
    have
    \begin{equation}\label{eq:normal.U.plus}
        0= d(X (b_{i,k}))-d(X (a_{i,k}))=\int_{a_{i,k}}^{b_{i,k}}\, e_N \cdot
        b_i(X (s),t-s,\alpha_i(s) \big)\ds\; .  
    \end{equation}
    By the regularity of $(b_i,c_i,l_i)$ with respect to $X(s)$ we have 
    $$
    \int_{a_{i,k}}^{b_{i,k}}\,  (b_i,c_i,l_i)\big(\xi_s,\alpha_i (s) \big)\ds
    =\int_{a_{i,k}}^{b_{i,k}}\,  (b_i,c_i,l_i)\big(x,t,\alpha_i (s)
    \big)\ds+O(\eta_k)(b_{i,k}-a_{i,k})\;.  
    $$
    Then, using the convexity of the images of $\BCL_i$, there exists a control $a^\flat_{i,k}$ such
    that 
    $$
    \int_{a_{i,k}}^{b_{i,k}}\,  (b_i,c_i,l_i)\big(\xi_s,a (s) \big)\ds =({b_{i,k}}-{a_{i,k}})\,
    (b_i,c_i,l_i)\big(x,t,\alpha^\flat_{i,k} \big)\ds+O(\eta_k)({b_{i,k}}-{a_{i,k}})\;, 
    $$
    and \eqref{eq:normal.U.plus} implies that $b_i\big(x,t,\alpha^\flat_{i,k} \big)\cdot
    e_N=O(\eta_k)$.  In terms of $\BCL$, this means we have a
    $(b^\flat_i,c^\flat_i,l^\flat_i)\in\BCL_i(x,t)$ such that $b^\flat_i\cdot e_N=O(\eta_k)$. 

    Using the normal controllabilty and regularity properties of $\BCL_i$, for $\eta_k$ small
    enough, there exists a $(b^\sharp_i,c^\sharp_i,l^\sharp_i)\in\BCL_i(x,t)$ which is
    $O(\eta_k)$-close to $(b^\flat_i,c^\flat_i,l^\flat_i)$ such that $b^\sharp_i\cdot e_N=0$.  
    This means that there exists a control $\alpha^\sharp_{i,k}\in A_i$ such that still
    $$
    \int_{a_{i,k}}^{b_{i,k}}\,  (b_i,c_i,l_i)\big(\xi_s,a (s)\big)\ds =({b_{i,k}}-{a_{i,k}})\,
    (b_i,c_i,l_i)\big(x,t,a^\sharp_{i,k} \big)\ds+O(\eta_k)({b_{i,k}}-{a_{i,k}}) 
    $$
    holds, and $b_i\big(x,t,a^\sharp_{i,k} \big)\cdot e_N=0$. 
    In other words, this specific control provides a regular dynamic.

    Hence, using the regularity of $\phi$, since $a^\sharp_{i,k}$ is regular we get
    $$
	\begin{aligned}
        \int_{a_{i,k}}^{b_{i,k}}&  A[\phi](s)\ds = 
        (b_{i,k}-a_{i,k})\Big\{\phi_t (x,t) - b_i(x,t,a^\sharp_{i,k}) \cdot D_x \phi(x,t)\\
        & \qquad\qquad\qquad +c(x,t,a^\sharp_{i,k})\phi(x,t)- l_i(x,t,a^\flat_{i,k}) + 
        O(\eta_k)\Big\}\;,\\[2mm]
        & \leq (b_{i,k}-a_{i,k})\Big\{ \partial_t \phi(x,t)+\HTreg \big(x,t, (\VFp)_*(x,t), 
        D_\H\phi(x,t) + O(\eta_k) \Big\}<0\; .
    \end{aligned}
    $$

    Therefore, for $\eta_k$ small enough, on each connected component of $\E_1$, $\E_2$ and on
    $\E_\H$, the integral is strictly negative and we get the desired contradiction.  
\end{proof}

\begin{remark}Notice that the alternative above with $\HTreg$ only holds
    for $\VFp$, and not for any arbitrary supersolution---see
    Theorem~\ref{thm:minimal.charac} where $\HT$ is used and not $\HTreg$.
\end{remark}

\section{Maximality of $\VFp$}

In order to prove that $\VFp$ is the maximal subsolution, we need the following result on
subsolutions 
\begin{lemma}\label{subsol-H}
    Assume that the ``standard assumptions in the codimension-$1$ case'' are satisfied.
    If $u:\R^N\times (0,\Tf)\to \R$ is an \usc subsolution of \eqref{pb:half-space}, then it
    satisfies 
    \begin{equation}\label{uH61}
        u_t+\HTreg (x,t,u,D_T u)\leq 0\quad \hbox{on  }\H\times (0,\Tf)\;.
    \end{equation}
\end{lemma}

\begin{proof}
   Let $\phi$ be a $C^1$--test-function on $\H\times (0,\Tf)$. Using the decomposition of $x\in
    \R^N$ in $(x',x_N)$ with $x'\in \R^{N-1}$, we can assume that $\phi$ is just a function of $x'$
    and $t$, and we can see $\phi$ as a function defined in $\R^N \times (0,\Tf)$ as well.

    If $(\xb,\tb)\in \H\times (0,\Tf)$ is a strict local maximum point of $u(x,t)-\phi(x',t)$ on
    $\H\times (0,\Tf)$, we have to show that 
    $$
        \phi_t(\xb',\tb)+ \HTreg(\xb,\tb,u(\xb,\tb), D_T \phi(\xb,\tb))\leq 0\; ,
    $$ 
    where $D_T \phi (\xb,\tb)$ is nothing but $D_{x'} \phi (\xb',\tb)$ and we
    also identify it below with the vector $(D_{x'} \phi (\xb',\tb),0)$.
    So, setting $a=\phi_t(\xb',\tb)$ and $p_T =D_T \phi (\xb,\tb)$, we have to prove that
    for any $(b,c,l) \in \BCL_T^\reg(\xb,\tb)$,
    $$\mathcal{I}:=a-b\cdot p_T +c u(\xb,\tb)-l \leq 0 \;.$$

    By definition of $\BCL_T^\reg(\xb,\tb)$, we can write
    $$(b,c,l) =\mu_1 (b_1,c_1,l_1) +\mu_2 (b_2,c_2,l_2)\; ,$$
    with $b_1\cdot e_N\leq0\leq b_2\cdot e_N$ and $\mu_1+\mu_2=1$. 
    Using the normal controllability and an easy
    approximation argument, we can assume without loss of generality that $b_1\cdot e_N< 0 <
    b_2\cdot e_N$. Of course, even if we do not write it to have simpler notations, $(b_1,c_1,l_1)$
    is associated to a control $\alpha_1$ and $(b_2,c_2,l_2)$ to a control $\alpha_2$.

    For $i=1,2$, we consider the affine functions
    $$ \psi_i (\delta):= a-b_i\cdot (p_T+\delta e_N)+c_i u(\xb,\tb)-l_i \; .$$
    By the above properties we have: $(i)$ $\psi_1$ is strictly increasing; $(ii)$ $\psi_2$ is
    strictly decreasing; $(iii)$ $\mu_1\psi_1(\delta)+\mu_2\psi_2(\delta)=\mathcal{I}$, which is
    independent of $\delta$.

    We argue by contradiction, assuming that $\mathcal{I}>0$ and choose $\bar \delta$ such that
    $\psi_1(\bar \delta)=\psi_2(\bar \delta)$. Notice that this is possible due to the strict
    monotonicity properties and the fact that $\psi_1(\R)=\psi_2(\R)=\R$. 
    We have therefore $\psi_1(\bar \delta)=\psi_2(\bar \delta)=\mathcal{I}>0$.

    \smallskip

    Next, for $0<\e \ll 1$, we consider the function
    $$ (x,t) \mapsto u(x,t)-\phi(x',t)-\bar \delta x_N -\frac{x_N^2}{\e^2}\;,$$
    defined in $\R^N\times(0,\Tf)$.
    Since $(\xb,\tb)$ is a strict local maximum point of $u-\phi$ on $\H\times (0,\Tf)$, there
    exists a sequence $(\xe,\te)$ of local maximum point of this function which converges to
    $(\xb,\tb)$, with $u(\xe,\te)$ converging to $u(\xb,\tb)$. 

    Our aim is to show that none of the $H_1$ or $H_2$ viscosity inequality holds for $u$ on $\H$,
    which will contradict the fact that $u$ is a viscosity subsolution. Assume for instance that
    the the $H_1$-inequality holds. Then $(\xe)_N \geq 0$ and by the regularity of $\phi$, 
    $$
    a-b_1(\xe,\te,\alpha_1)\cdot (p_T +\bar \delta e_N
    +\frac{2(\xe)_N}{\e^2}e_N)+c_1(\xe,\te,\alpha_1) u(\xb,\tb)-l_1(\xe,\te,\alpha_1) \leq  o_\e
    (1)\; .  
    $$
    But since $(\xe,\te)\to (\xb,\tb)$, $b_1(\xe,\te,\alpha_1)\to b_1(\xb,\tb,\alpha_1)$ and
    therefore $b_1(\xe,\te,\alpha_1)\cdot e_N < 0$ for $\e$ small enough. Using that
    $(\xe)_N \geq 0$, this inequality implies 
    $$
        a-b_1(\xe,\te,\alpha_1)\cdot (p_T +\bar \delta e_N)+c_1(\xe,\te,\alpha_1)
        u(\xb,\tb)-l_1(\xe,\te,\alpha_1) \leq  o_\e (1)\;. 
    $$
    By the definition and properties of $\bar \delta$ and the fact that $\mathcal{I}>0$, this
    inequality cannot hold for $\e$ small enough, showing that the $H_1$ inequality cannot hold
    neither.  A similar argument being valid for the $H_2$ inequality, we have a contradiction and
    therefore $\mathcal{I}\leq 0$, and the proof is finished. 
\end{proof}

\begin{theorem}\label{thm:Vp-max}\emph{--- Maximality of $\VFp$.}\smsp 
    Assume that the ``standard assumptions in the codimension-$1$ case'' are satisfied. Then 
	$\VFp$ is continuous and it is the maximal Ishii solution of \eqref{pb:half-space}.
\end{theorem}

\begin{proof}
    Let $u$ be any subsolution of \eqref{pb:half-space}. We want to show that $u\leq (\VFp)_*$ in
    $\R^N \times [0,\Tf)$ and to do so we first notice that, as we did in the proof of the
    characterization of $\VFm$ (Theorem~\ref{thm:minimal.charac}), we can reduce the proof to a
    local comparison argument since \LOCa and \LOCb are satisfied. So, let $\cyl$ be a cylinder
    in which we want to perform the \LCR between $u$ and $(\VFp)_*$.

    Using again the arguments of the proof of Theorem~\ref{thm:minimal.charac}, we may assume
    without loss of generality that $u$ is a strict subsolution of \eqref{pb:half-space} and in
    particular a strict subsolution of \eqref{eq:HTregineqsub}. Finally we can regularize $u$ in
    order that it is $C^1$ on $\H\times (0,\Tf)$.

    Using Theorem~\ref{thm:sub.Qk.dpp.bt} to show that $u$ satisfies a sub-dynamic programming
    principle with trajectories in $\mT(x,t)$, we see that we are (almost) in the framework of
    Lemma~\ref{lem:comp.fundamental}, the usual $\F^{\mathcal{M}}$-inequality for $u$ being replaced
    by \eqref{uH61}.

    Using in an essential way Theorem~\ref{teo:condplus.VFp}\footnote{which replaces the arguments
    for the supersolution $v$ in the proof of Lemma~\ref{lem:comp.fundamental} (cf.
    Remark~\ref{rem:comp.fundamental}).}, it is easy to see that the result of
    Lemma~\ref{lem:comp.fundamental} still holds in this slightly different framework and yields
	$$
    \max_{\cylb}(u-(\VFp)_*)\leq\max_{\partial\cyl}(u-(\VFp)_*)\;,
    $$
	and the \GCR follows: $u\leq(\VFp)_*$ in $\R^N\times[0,\Tf]$.

    Concerning the continuity statement, consider $u=(\VFp)^*$. By definition,
    $(\VFp)^*\geq(\VFp)_*$ but the comparison result above applied to $(\VFp)^*$ which is a
    subsolution shows that in the end $\VFp=(\VFp)_*=(\VFp)^*$. Hence $\VFp$ is continuous and is
    maximal amongst Ishii subsolutions.  
\end{proof}

\section{Appendix: stability of regular trajectories}
\label{subsec:extraction}

This appendix is about proving the convergence property of regular trajectories,
Lemma~\ref{lem:limit.reg.traj}. We actually prove a more general result here:
\begin{proposition}\label{prop:extraction.gen}
	Let $t>0$ be fixed and for each $\eps>0$ let $\BCL^\eps$ be a set-valued map satisfying \HBCLa
	and let $(X,D,L)^\eps$ be solution of 
	the differential inclusion
	$$\forall s\in(0,t)\;,\quad (X,D,L)^\eps(s)\in\BCL^\eps(X^\eps(s),t-s)\;.$$
	\begin{enumerate}
		\item[$(i)$] If $\BCL^\eps$ converges to $\BCL$ locally uniformly in $\R^N\times(0,t)$ 
		(for the Hausdorff distance on sets) and $(X,D,L)^\eps(0)\to (x,d,l)$, then,
		up to extraction, $(X,D,L)^\eps$ converges to some trajectory $(X,D,L)$ which satisfies 
	$$\forall s\in(0,t)\;,\quad (X,D,L)(s)\in\BCL(X(s),t-s)\,$$
		with initial value $(X,D,L)(0)=(x,d,l)$.
		\item[$(ii)$] If moreover each trajectory $X^\eps$ is regular,
        then the limit trajectory $X$ is also regular.
	\end{enumerate}
\end{proposition}

This result is obtained through several lemmas. The first one proves part $(i)$ of the proposition,
which is not very difficult.  
\begin{lemma}\label{lem:extraction.simple}
	If $\BCL^\eps$ converges to $\BCL$ locally uniformly in $\R^N\times(0,t)$ 
	(for the Hausdorff distance on sets) and $(X,D,L)^\eps(0)\to (x,d,l)$, then
	up to extraction, $(X,D,L)^\eps$ converges to some trajectory $(X,D,L)$ which is a solution of 
	the differential inclusion associated with $\BCL$, with the corresponding initialization.
\end{lemma}
\begin{proof}
    Notice first that since the $\BCL^\eps$ all satisfy \HBCLa with constants independent of $\eps$,
    and the initial value converges, the trajectories $(X,D,L)^\eps$ are equi-Lipschitz and
    equi-bounded on $[0,t]$. Hence we can extract a subsequence $(X,D,L)^\epsn$ converging to 
    $(X,D,L)$ uniformly on $[0,t]$. Moreover, for any $\kappa>0$ small enough, if $n$ is big
    enough we have 
    $$\forall s\in(0,t)\;,\quad \BCL^\eps(X^\epsn(s),t-s)\subset \BCL(X(s),t-s)+\kappa B_{N+3}$$ 
    where $B_{N+3}$ is the unit ball of $\R^{N+3}$. Passing to the limit as $\epsn\to0$,
    we deduce that $(X,D,L)$ satisfies the differential inclusion associated with $\BCL$, and of
    course its initial data is $(X,D,L)(0)=(x,d,l)$.
\end{proof}

Now we need several results in order to prove part $(ii)$ which is much more involved.
Before proceeding, let us comment a little bit: using the control representation of 
the differential inclusion (Lemma~\ref{lem:struc.traj}), there exist some controls $\alpha_i^\eps,a^\eps$
such that
	$$
		 \dot X^{\eps}(s) =  \sum_{i=1,2} b_i^{\eps}\big(X^{\eps}(s), 
		 t-s,\alpha_i^{\eps}(s)\big)    
		 \mathds{1}_{\{X^{\eps}\in\Omega_{i}\}}(s)+ b_\H^{\eps} 
		 \big(X^{\eps}(s),t-s,a^{\eps}(s)\big){}
		 \mathds{1}_{\{X^{\eps}\in\H\}}(s)\;.
	$$
Recall that the control $a^\eps$ is actually complex since it involves $\alpha_1^\eps,\alpha_2^\eps$
but also $\alpha_0^\eps$.  In other words, $b_\H$ is a mix of $b_0,b_1,b_2$ with weights
$\mu_0^\eps,\mu_1^\eps,\mu_2^\eps$.  However, notice that focusing on regular dynamics, the
$b_0$-term is not a problem since it is already tangential (hence, regular).

In order to send $\eps\to0$ we face two difficulties: the first one is that we have to deal 
with weak convergences in the $b_i^{\eps}, b_\H^{\eps}$-terms. But the problem 
is increased by the fact that some pieces of the limit trajectory $X(\cdot)$ on $\H$ 
can be obtained as limits of trajectories $X^{\eps}(\cdot)$ which lie either on 
$\H$, $\Omega_{1}$ or $\Omega_{2}$. In other words, the indicator functions 
$\mathds{1}_{\{X^{\eps}\in\H\}}(\cdot)$ do not necessarily converge to 
$\mathds{1}_{\{X\in\H\}}(\cdot)$, and similarly the
$\mathds{1}_{\{X^{\eps}\in\Omega_{i}\}}(\cdot)$ do not converge to 
$\mathds{1}_{\{X\in\Omega_{i}\}}(\cdot)$. 

From Lemma \ref{lem:extraction.simple} we already know that $\dot X^\eps$ converges weakly on
$(0,t)$ to some $\dot X$ which can be represented as for $X^\eps$ above, by means of some controls
$(\alpha_1,\alpha_2,a)$.  The question is to prove that this control $a$ yields regular dynamics on
$\H$. In order to to do, we introduce several tools. The first one is a representation of $X$ by
means of some regular controls $(\alpha_1^\sharp,\alpha_2^\sharp,a^\sharp)$.  Those controls may
differ from $(\alpha_1,\alpha_2,a)$, but they are an intermediate step which will help us to prove
the final result.

\begin{lemma}\label{lem:extrac.measures}
    For any $s\in(0,t)$ there exists three measures $\nu_1(s,\cdot),\nu_2(s,\cdot),\nu_\H(s,\cdot)$
    on $A_1,A_2,A$ respectively and three controls
    $(\alpha_1^\sharp(s),\alpha_2^\sharp(s),a^\sharp(s))\in A_1\times A_2\times A$ such that
	\begin{enumerate}
		\item[$(a)$] $\nu_1,\nu_2,\nu_\H\geq0$,\ $\nu_1(s,A_1)+\nu_2(s,A_2)+\nu_\H(s,A)=1$\,;
		\item[$(b)$] up to extraction, 
            $b_1^\eps(X^\eps(s),t-s,\alpha_1^\eps)\to b_1(X(s),t-s,\alpha_1^\sharp(s))\cdot
            \nu_i(s,A_1)$\;,\\ and the same holds for $b_2,b_\H$ with measures $\nu_2,\nu_H$ and
            controls $\alpha^\sharp_2,\alpha^\sharp_\H$\,;
        \item[$(c)$] for $i=1,2$, $b_i(X(s),t-s,\alpha^\sharp_i(s))\cdot e_N=0 \quad
            \nu_i\text{-a.e. on }\{X(s)\in\H\}\;.$ 
    \end{enumerate}
	In particular, the dynamic obtained by using $(\alpha_0,\alpha_1^\sharp,\alpha_2^\sharp)$ is regular.
\end{lemma}
\begin{proof}
	We use a slight modification of the procedure leading to relaxed control 
	as follows. We write
	$$
	b_1^{\eps}\big(X^{\eps}(s),t-s,\alpha^{\eps}_1(s)\big) 
		 \mathds{1}_{\{X^{\eps}\in\Omega_{1}\}}(s)=
		 \int_{A_{1}}b_1^{\eps} 
		 \big(X^{\eps}(s),t-s,\alpha\big)\,
		 \nu_{1}^{\eps}(s,\dalpha)\;,
	$$
	where $\nu_{1}^{\eps}(s,\cdot)$ stands for the measure defined on 
	$A_{1}$ by	$\nu_{1}^{\eps}(s,E)=\delta_{\alpha_1^{\eps}}(E)
	\mathds{1}_{\{X^{\eps}\in\Omega_{1}\}}(s)$, for any Borelian set 
	$E\subset A_{1}$. Similarly we define $\nu_{2}^{\eps}$ and 
	$\nu_{\H}^{\eps}$ for the other terms. Notice that $\nu_{\H}^{\eps}$ is 
	a bit more complex measure since it concerns controls of the form 
	$a=(\alpha_{1},\alpha_{2},\mu)$ on $A$, but it works as for $\nu_{1}^\eps$ so we omit 
	the details. 
	
    Note that, for any $s$, $\nu^{\eps}_{1}(s,A_1)+\nu^{\eps}_{2}(s,A_2)+\nu^{\eps}_{\H}(s,A)=1$ and
    therefore the measures $\nu_{1}^{\eps}(s,\cdot), \nu_{2}^{\eps}(s,\cdot),
    \nu_{\H}^{\eps}(s,\cdot)$ are uniformly bounded in $\eps$. Up to successive extractions of
    subsequences, they all converge in $L^\infty(0,\Tf;E)$ weak-$*$ (where $E=A_1,A_2,A$) 
    to some measures $\nu_{1}$, $\nu_{2}$, $\nu_{\H}$. Since moreover
    the total mass is $1$, we obtain in the limit $\nu_{1}(s,A_1)+\nu_{2}(s,A_2)+\nu_{\H}(s,A)=1$.  

    Using that up to extraction $X^{\eps}$ converges uniformly on $[0,t]$, using the local uniform
    convergence of the $b_{1}^{\eps}$, we get that
	$$
		\int_{A_{1}}b_1^{\eps} 
		 \big(X^{\eps}(s),t-s,\alpha\big)\,
		 \nu_{1}^{\eps}(s,\dalpha)
		 \ \mathop{\longrightarrow}_{\eps\to 0} \ 
		\int_{A_{1}}b_1 
		 \big(X(s),t-s,\alpha\big)\,
		 \nu_{1}(s,\dalpha),
	$$
    weakly in  $L^\infty(0,\Tf)$. Introducing $\pi_{1}(s):= \int_{A_1}  \nu_{1}(s,\dalpha)$ and
    using the convexity of $A_{1}$ together with a measurable selection argument (see \cite[Theorem
    8.1.3]{AF}), the last integral can be written as
    $b_{1}\big(X(s),\sigma(s),\alpha_{1}^{\sharp}(s)\big)\pi_{1}(s)$ for some control
	$\alpha^{\sharp}_{1}\in L^{\infty}(0,\Tf;A_{1})$. 
	The same procedure for the 
	other two  terms provides the controls $\alpha^{\sharp}_{2}(\cdot)$, 
	$a^{\sharp}(\cdot)$ and functions $\pi_{2}(\cdot),\ \pi_{\H}(\cdot)$, which yields $(a)$ and $(b)$.

	We now turn to property $(c)$ that we prove for $b_1$, the proof being identical for $b_2$.
	Since $(X_N^\eps)_+:=\max(X_N^\eps,0)$ is a sequence of Lipschitz 
	continuous functions which converges uniformly to $(X_N)_+$ on $[0,t]$, up to 
	an additional extraction of subsequence, we may assume that the derivatives converge
	weakly in $L^\infty$ (weak--$*$ convergence). As a consequence,
	$\frac{d}{ds}\big[(X_N^\eps)_+\big]\mathds{1}_{\{X\in\H\}}$ converges 
	weakly to 
	$\frac{d}{ds}\big[(X_N)_+\big]\mathds{1}_{\{X\in\H\}}$.
	
	By Stampacchia's  Theorem  we have 
	$$
	\frac{d}{ds}\big[(X_N^\eps)_+ \big]=\dot X^{\eps}_N(s)\,
		\mathds{1}_{\{X^{\eps}\in\Omega_1\}}(s)  \quad \mbox{ for almost all } s \in (0,t).
	$$
	Therefore, the above convergence reads, in $L^{\infty}(0,\Tf)\text{weak--}*$
	$$ 
		\dot X^{\eps}_N(s)
		\mathds{1}_{\{X^{\eps}\in\Omega_1\}}(s)\mathds{1}_{\{X\in\H\}}(s)	
		\longrightarrow
		\dot X_N(s)\mathds{1}_{\{X\in\Omega_1\}}(s) 
		\mathds{1}_{\{X\in\H\}}(s)=0\;.
	$$
	Using the expression of $\dot X^{\eps}(s)$, 
	$\big(b_1^{\eps}\big(X^{\eps}(s), 
	t-s,\alpha_1^{\eps}(s)\big)\cdot e_N\big)\mathds{1}_{\{X^{\eps}\in\Omega_1\}}(s)
	\mathds{1}_{\{X\in\H\}}(s) \rightarrow 0$ in $L^{\infty}(0,\Tf)\text{ 
	weak--}*$ which implies that
	\begin{equation}\label{eq:regu.dyn}
        \Big(b_1 \big(X(s),t-s,\alpha_1^{\sharp}(s)\big)\cdot e_N \Big)\,\pi_{i}(s)=0\; \hbox{ a.e.
        on  }\{X(s)\in \H\}\;, 
    \end{equation}
    which yields property $(c)$. This means that $b_i(X(s),t-s,\alpha_i^\sharp(s))$ is tangential on
    $\H$ so that combining them with some $b_0$ (which is tangential by definition), we get a
    regular dynamic on $\H$.
\end{proof}

We now want to prove that the controls $(\alpha_1,\alpha_2,a)$ yield regular strategies, not only 
the $(\alpha_1^\sharp,\alpha_2^\sharp,a^\sharp)$.
In order to proceed we introduce the set of regular dynamics: 
$$\forall (z,s)\in\H\times[0,t]\;,\quad K(z,s):=\big\{b_{\H}\big(z,s,a_{*}\big)\;, 
a_{*}\in \Aoreg(z,s)\big\}\subset\R^{N}\;.$$
We notice that, for any $z\in\H$ and $s\in[0,\Tf]$, $K(z,s)$ is closed 
and convex, and the mapping $(z,s)\mapsto K(z,s)$ is continuous 
on $\H$ for the Hausdorff distance. 
Then, for any $\eta>0$, we consider the subset of $[0,t]$ 
consisting of times $s$ for which one has singular ($\eta$-enough) dynamics for 
the control $a(\cdot)$, namely
$$\begin{aligned}
	\Esing&:=\bigg\{s\in [0,t]: X(s)\in\H\text{ and } 
	\dist\Big(b_{\H}\big(X(s),t-s,a(s)\big);K\big(X(s),t-s\big)\Big)\geq\eta \bigg\}\;.
	\end{aligned}
$$
If $s\in\Esing\neq\emptyset$, since $K(X(s),t-s)$ is closed and convex, 
there exists an hyperplane separating $b_{\H}\big(X(s),t-s,a(s)\big)$ from 
$K(X(s),t-s)$ and we can construct an affine function $\Psi_{s}:\R^{N}\to\R$ of the form
$\Psi_s(z)=\zeta_1(s)z+\zeta_2(s)$ such that
$$
	\Psi_{s}\bigg(b_{\H}\big(X(s),t-s,a(s)\big)\bigg)\geq 1 \; \hbox{if  }s\in \Esing\;,\quad
	\Psi_{s}\leq 0 \text{ on }K\big(X(s),t-s\big) \;.
$$
In other words, $\Psi_s$ ``counts'' the singular dynamics.

Since the mapping $s\mapsto b_{\H}\big(X(s),t-s,a(s)\big)$ is measurable and 
$s\mapsto K\big(X(s),t-s\big)$ is continuous, we can assume that 
$s\mapsto \zeta_1(s),\zeta_2(s)$ are measurable and bounded (because the distance $\eta>0$ 
is fixed), which allows to define the quantity
	$$
	I(\eta):=\begin{cases}
        \displaystyle\int_{0}^{t}\big(\Psi_{s}(\dot X(s)\big)\mathds{1}_{\Esing}(s)\ds & 
        \text{if }\Esing\neq\emptyset\\ 0 & \text{if }\Esing=\emptyset\;.
	\end{cases}
	$$
By definition, it is clear that $I(\eta)\geq|\Esing|$ (the Lebesgue measure of $\Esing$).
The following result gives a converse estimate

\begin{lemma}\label{lem:extrac.Esing}
	For any $\eta>0$, $I(\eta)\leq0$.
\end{lemma}
\begin{proof}
    Let $\eta>0$. If $\Esing=\emptyset$ there is nothing to do so let us assume that this is not the
    case, and take some $s\in\Esing$.	Since $\Psi_s$ is affine, using the weak convergence of
    $\dot X^\eps$ we know that
	$$
    I(\eta)=\lim_{\eps\to0}I^{\eps}(\eta):=\int_{0}^{t}\big(\Psi_{s}(\dot
    X^{\eps}(s)\big)\mathds{1}_{\Esing}(s)\ds\;. 
    $$
    The strategy is to use Lemma~\ref{lem:extrac.measures} to pass to the limit and estimate
    $I^\eps(\eta)$, knowing that at each level $\eps>0$, the dynamics are regular. In order to keep
    this information in the limit, dealing with the $b_i^\eps$-terms is handled by property (c) of
    Lemma~\ref{lem:extrac.measures}. But the $b_\H^\eps$-term is more delicate: we need first to fix
    a regular control independent of $\eps$.

    To do so, we start by noticing that for fixed $\eps>0$ and $s \in [0,t]$, for each  $a^{\eps}(s)
    \in A^{\rm reg}_0(X^\eps(s), t-s)$  there exists a $\tilde{a}^\eps(s)   \in A^{\rm reg}_0(X(s),
    t-s)$ such that 
	$$
	  b_\H^{\eps}(X^{\eps}(s),t-s,a^{\eps}(s))=b_\H(X(s),t-s,\tilde{a}^{\eps}(s))+o_\eps(1)\;.
	$$
    Indeed, this comes from a measurable selection argument and the fact that $X^{\eps}$ converges
    uniformly to $X$, while $b^\eps_\H$ also converges locally uniformly (with respect to its first
    variable). So, rewriting the expansion of $\dot X^\eps$ and using that $\Psi_s$ is affine we get
	$$\begin{aligned}
		I^\eps(\eta) =& \int_0^t
		 \Psi_s\Big(\sum_{i=1,2} b_i^{\eps}\big(X^{\eps}(s), 
		 t-s,\alpha_i^{\eps}(s)\big)    
		 \mathds{1}_{\{X^{\eps}\in\Omega_{i}\}}(s)\Big)\mathds{1}_{\Esing}(s)\ds\\
		&+ \int_0^t \zeta_1(s)\Big(b_\H \big(X(s),t-s,\tilde a^{\eps}(s)\big)\,
		 \mathds{1}_{\{X^{\eps}\in\H\}}(s)\Big)\mathds{1}_{\Esing}(s)\ds +o_\eps(1)\;.
	\end{aligned}$$

    Moreover, by construction and using again a measurable selection argument (see Filippov's Lemma
    \cite[Theorem 8.2.10]{AF}), there exists a control $a_{\star}(s) \in K(X(s),t-s)$ such that 
	$$
	 \zeta_1(s)b_\H(X(s),t-s,a_{\star}(s)) = \max_{a\in K(X(s),t-s) }  \zeta_1(s)b_\H(X(s),t-s,a).
	$$ 
	Therefore,
	$$\begin{aligned}
		I^\eps(\eta) \leq \int_0^t
		 \Psi_s\Big\{ \sum_{i=1,2} & b_i^{\eps}\big(X^{\eps}(s), 
		 t-s,\alpha_i^{\eps}(s)\big)    
		 \mathds{1}_{\{X^{\eps}\in\Omega_{i}\}}(s)\\
		+ & b_\H \big(X(s),t-s, a_\star(s)\big)\,
		 \mathds{1}_{\{X^{\eps}\in\H\}}(s)\Big\}\mathds{1}_{\Esing}(s)\ds +o_\eps(1)\;.
	\end{aligned}$$

    Now we pass to the weak limit, using Lemma~\ref{lem:extrac.measures} but with a constant $b_\H$
    instead of $b^\eps_\H$ and, more importantly, a constant control $a_\star$. In other words, the
    measure $\nu_\H^\eps$ is actually independent of $\eps$ in this situation.  We get some measures
    $\nu_1,\nu_2,\nu_\H$ and some controls $\alpha^\sharp_1,\alpha^\sharp_2$ and $a^\sharp=a_*$
    here, for which
	$$\begin{aligned}
		\lim_{\eps\to0} I^\eps(\eta) \leq \int_0^t
		 \Psi_s\Big\{ \sum_{i=1,2} & b_i\big(X(s), 
		 t-s,\alpha_i^{\sharp}(s)\big)\nu_i(s,A_i)\\
		+ & b_\H \big(X(s),t-s, a_\star(s)\big)\nu_H(s,A)\,
		\Big\}\mathds{1}_{\Esing}(s)\ds\;.
	\end{aligned}$$
    Recall that by construction $ b_\H(X(s),t-s,a_{\star}(s)) \in  K(X(s),t-s)$ and that
    $\alpha^\sharp_1,\alpha^\sharp_2$ are regular controls. Therefore, since
    $\nu_{1}(s,A_1)+\nu_{2}(s,A_2)+\nu_{\H}(s,A)=1$ and the set $K(X(s),t-s)$ is convex, we deduce
    that	the convex combination satisfies 
    $$\Psi_s\Big\{ \sum_{i=1,2} b_i\big(X(s), 
		 t-s,\alpha_i^{\sharp}(s)\big)\nu_i(s,A_i)
		+ b_\H \big(X(s),t-s, a_\star(s)\big)\nu_H(s,A)\,
		\Big\}\leq 0\;.$$
	The conclusion is that $I(\eta)=\lim_{\eps\to0} I^\eta(\eta)\leq0$ and the result is proved.
\end{proof}

\begin{proof}[Proof of Proposition~\ref{prop:extraction.gen}]
    The first part $(i)$ is done in Lemma~\ref{lem:extraction.simple}.  As for $(ii)$, we proved
    above that for any $\eta>0$, $|\Esing|\leq I(\eta)=0$, so that set $\Esing$ is of zero Lebesgue
    measure.  Hence, using a countable union of negligeable sets we deduce that
    $$\begin{aligned}
	    \bigg\{s\in [0,t]: X(s)\in\H\text{ and } 
	    b_{\H}\big(X(s),t-s,a(s)\big)\notin K\big(X(s),t-s\big)\Big)\bigg\}\,
	\end{aligned}$$
    is also of zero Lebesgue measure. This means that for almost any $s\in(0,t)$, the strategy
    obtained by choosing $a$ as control is regular, which concludes the proof.
\end{proof}

\chapter{Uniqueness and Non-Uniqueness Features}
\label{sect:uniq.ishii}

\abstract{This chapter is devoted to a discussion of the uniqueness and the non-uniqueness
properties for the Ishii solutions of the standard HJB Equation, \ie we investigate the cases when 
the value functions $\VFm$ and $\VFp$ are equal and when they are different.
Counter-examples to uniqueness are given but also conditions on the Hamiltonians
$\HT$ and $\HT^{\rm reg}$---the Hamiltonians of the additional subsolution inequalities for $\VFm$ and
$\VFp$ on the hyperplane---ensuring that they coincide, leading to a pure pde characterization of the
uniqueness cases.}

In this chapter, we investigate the question of the uniqueness for Ishii solutions of
Problem~\eqref{pb:half-space}, which can be summarized as: when are the value functions $\VFp$,
$\VFm$ equal? It is rather clear that, in general, they are different since the restriction to use
only regular controls can really penalize the controller, leading to the fact that $\VFp$ is
strictly larger than $\VFm$. We give an example of this non-uniqueness situation in the first
section of this chapter.

Then we provide some conditions under which uniqueness holds, using a pde point-of-view: as a
consequence of Theorem~\ref{thm:minimal.charac} and Proposition~\ref{prop:ishii-Up}, we know that
$\VFp=\VFm$ if $\HT=\HT^\reg$, and we give a simple condition under which this last equality is true.

\section{A typical example where $\VFp\not\equiv\VFm$} 

We consider a one-dimensional finite horizon problem where
$$\Omega_1=\{x>0\},\ \Omega_2=\{x<0\},\ \H=\{x=0\}\; .$$
The reader will find in \cite{BBC1} a detailed study of this situation for infinite horizon control
problems, a general description of the structure of solutions, the link between the minimal and
maximal Ishii solutions with state-constraints solutions as well as several explicit examples.
Here we restrict ourselves to exposing an explicit example of non-uniqueness for illustration
purposes.

We consider the dynamics
$$
    \dot X(t)=\alpha_1(t)\text{ in }\Omega_1\;,\quad\dot X(t)=\alpha_2(t)\text{ in }\Omega_2\;,
$$
where $\alpha_1(\cdot),\alpha_2(\cdot)\in L^\infty\big(0,+\infty;[-1,1]\big)$ are the controls. In
other words, $A_1=A_2=[-1,1]$ and $b_1(x,t,\alpha_1)=\alpha_1$, $b_2(x,t,\alpha_2)=\alpha_2$.  As
for the costs, we choose
$$
    l_1(x,t,\alpha_1)=1-\alpha_1+\min(|x|,1) \text{ in }\Omega_1\;,\quad
    l_2(x,t,\alpha_2)=1+\alpha_2+\min(|x|,1) \text{ in }\Omega_2\;.
$$
Finally, we set $c_1(x,\alpha_1)=c_2(x,\alpha_2)=1$ for the discount factor and also $g=\min(|x|,1)$
for the final cost. Therefore, 
$$
    \VFm(x,t)=\inf_{\mT(x,t)}\left\{\int_0^{t}l(X(s),t-s,a(s))e^{-s}\ds + g(X(t))e^{-t}\right\} \;,
$$
where $l$ is either $l_1$, $l_2$ or a convex combination of both for $x=0$, and
$a(\cdot)=(\alpha_1,\alpha_2,\mu)$ is the extended control. The definition for $\VFp$ is similar,
the infimum being taken over $\mT^\reg(x,t)$.

\noindent\textbf{Computing $\VFm(0,t)$.} 
It is clear that $l_1(x,\alpha_1), l_2(x,\alpha_2)\geq 0$ and these running costs are even strictly
positive for $x\neq 0$. Therefore, $\VFm(x,t) \geq 0$ for any $x\in \R$ and $t\geq 0$. On
the other hand, for $x=0$, we have access to a $0$-cost strategy by choosing the singular
``pull-pull'' strategy $a=(\alpha_1,\alpha_2,\mu)=(1,-1,1/2)$ which gives
$$ b (0,t-s,a) = \mu \alpha_1+(1-\mu)\alpha_2=0\; ,$$
$$ l (0,t-s,a)  = \mu(1-\alpha_1)+(1-\mu)(1+\alpha_2)=0\; .$$
As a consequence, it is clear that this is the best strategy for $x=0$ and $\VFm(0,t)=0$ for any
$t\geq 0$.

\noindent\textbf{Computing $\VFp(0,t)$}.
For simplicity, we compute it only for $t\leq 1$ here. In this
case any trajectory satisfies $|X(s)| \leq 1$ for any $0\leq s \leq t$ and $\min(|X(s)|,1)$ can be
replaced by $|X(s)|$ everywhere (in the running cost and terminal cost).

If $X$ is any trajectory starting from $X(0)=0$ and associated to a regular control and if
$X(s)>0$, then 
$$\begin{aligned}
    l(X(s),t-s,a(s))e^{-s}= &(1 -\dot X(s)+X(s))e^{-s} \\
    =& e^{-s}-  (X(s)e^{-s})'\;.
\end{aligned}$$
With analogous computations for $X(s)<0$, we end up with
$l(X(s),t-s,a(s))e^{-s}=e^{-s}- [|X|e^{-s}] '(s)$ if $X(s)\neq0$. 

It remains to examine the case when $X(s)=0$. It is easy to see that, if $b(0,t-s,a)=0$ is a regular
dynamic, then $l(0,t-s,a)\geq 1$ since $\alpha_1\leq 0, \alpha_2\geq 0$ and $l(0,t-s,a)= 1$ if and
only if $\alpha_1=\alpha_2= 0$. Therefore, for $X(s)=0$, the above formula is changed into
$l(X(s),t-s,a(s))e^{-s}\geq e^{-s}- [|X|e^{-s}] '(s)$ since $|X'(s)|=0$ a.e. on the set
$\{X(s)=0\}$. And actually, equality is attained for the above mentioned choice of $a$. Therefore
$$\begin{aligned}
    \int_0^{t}l(X(s),t-s,a(s))e^{-s}\ds + g(X(t))e^{-t}= &
    \int_0^{t}\left(e^{-s}- [|X|e^{-s}] '(s)\right)\ds + g(X(t))e^{-t}\\
    = & 1-e^{-t}>0\; ,
\end{aligned}$$
proving that $\VFp(0,t)=1-e^{-t}>\VFm(0,t)=0$ at least for $0<t\leq 1$. 
The conclusion is that $\VFp\neq\VFm$ and uniqueness does not
hold in the class of Ishii solutions.

\section{Equivalent definitions for $\HT$ and $\HT^\reg$}\label{upoH}

We recall that we defined $\HT$ and $\HT^\reg$ in Section~\ref{subsec:complemented}, using the
subsets $\BCL_T(x,t)$ and $\BCL_T^\reg(x,t)$: for $x \in \H$, $t\in (0,\Tf)$, $r\in\R$, $p\in\R^N$
\begin{equation}\label{def:HT.app}
\HT(x,t,r,p):=\sup_{(b,c,l)\in\BCL_T(x,t)}\big\{-b\cdot p+cu-l\big\}\;,
\end{equation}
while the second Hamiltonian is defined similarly but by
considering only \emph{regular} tangential dynamics~$b$
\begin{equation}\label{def:HTreg.app}
\HT^\reg(x,t,r,p):=\sup_{\BCL_T^\reg(x,t)}\big\{-b\cdot p+cu-l\big\}\;.
\end{equation}

On the other hand, for any $x,t,r,p'$, the functions $f(s):=H_1(x,t,r,p'+se_N)$ and
$g(s):=H_2(x,t,r,p'+se_N)$ are convex and, thanks to Section~\ref{sect:quasi.convexity}, we can
introduce the nonincreasing and nondecreasing parts
$f^\sharp,f_\flat, g^\sharp,g_\flat$ of $f$ and $g$. It is easy to see that
$$f^\sharp(s)=\sup_{(b_1,c_1,l_1)\in\BCL_1(x,t) \atop b_1\cdot e_N\leq0}\big\{-b_1\cdot (p'+se_N)+c_1u
-l_1\big\}:=H_1^-(x,t,r,p'+se_N)\;,$$
and similarly we define ``$H_2^-=g^\sharp$'', ``$H_1^+=f_\flat$''  and ``$H_2^+=g_\flat$'' , the
choice of ``$+$'' or ``$-$'' in $H_i^\pm$ being related to the sign of $b_i\cdot e_N$ in its definition.

In order to provide equivalent definitions of $\HT$,$\HT^\reg$, we follow
Section~\ref{sect:quasi.convexity} where we introduced $M(s):=\max(f(s),g(s))$ and
$M^\reg(s):=\max(f^\sharp(s),g_\flat(s))$, which leads to consider the Hamiltonians defined for $x
\in \H$, $t\in (0,\Tf)$, $r\in\R$, $p \in \R^N$ by 
\begin{equation}  \label{defHti}
\Hti(x,t,r,p):= \max\big( H_1(x,t,r,p), H_2(x,t,r,p)\big )\;,
\end{equation}
\begin{equation}   \label{defHtireg}
\Htireg(x,t,r,p):= \max\big(H_{1}^- (x,t,r,p),H_{2}^+ (x,t,r,p)\big)\; .
\end{equation}

The following representation holds
\begin{lemma}\label{lem:H1m.H2p.a}
    For any $(x,t,r,p')\in\H\times(0,\Tf)\times\R\times\R^N$,
    \begin{equation} \label{HT1}
    \HT(x,t,r,p')= \min_{s\in\R} \Hti(x,t,r,p^\prime+s e_N)\;,
    \end{equation}
    \begin{equation} \label{HT2}
        \HTreg(x,t,r,p')= \min_{s\in\R} \Htireg(x,t,r,p'+s e_N)\;.
    \end{equation}
    Moreover, there exist $\nu_1\leq\nu_2$ such that for any $\lambda\in[\nu_1,\nu_2]$,
    \begin{equation}\label{HT3}
        \HTreg(x,t,r,p')=H_1^-(x,t,r,p'+\lambda e_N)=H_2^+(x,t,r,p'+\lambda e_N)\;.
    \end{equation}
\end{lemma}

\begin{proof}
    Notice first that \eqref{HT3} is a direct consequence of Lemma~\ref{lem:qcr.Mreg}. 
    Now, concerning \eqref{HT1} and \eqref{HT2}, We only provide the full proof in the case of $
    \HT$, the one for $\HTreg$ follows from the same arguments, just changing the sets of
    $(b_1,c_1,l_1)$, $(b_2,c_2,l_2)$ we consider.

    We introduce the function $\varphi:\R\to \R$ defined by
    $$
        \varphi (s) := \max(H_1(x,t,r,p'+se_N) ,H_2(x,t,r,p'+se_N))\; .
    $$
    This function is convex, continuous and coercive since both $H_1,H_2$ have these properties and
    therefore there exists $\ovs \in \R$ such that $\varphi (\ovs)=\min_{s\in \R}\,\varphi (s)$. As a
    consequence, $0\in \partial \varphi (\ovs)$, the convex subdifferential of $\varphi$.

    We apply a classical result on the subdifferentials of convex functions defined as supremas of
    convex (or $C^1$) functions (cf \cite{R}): here 
    $$
    \varphi(s)=\sup\left\{-b_1\cdot(p'+se_N)+c_1r-l_1;-b_2\cdot(p'+se_N)+c_2r-l_2\right\}\; ,$$
    where the supremum is taken over all $(b_1,c_1,l_1)\in \BCL_1(x,t)$ and $(b_2,c_2,l_2)\in
    \BCL_2(x,t)$.

    The functions $s\mapsto -b_i\cdot(p'+se_N)+c_ir-l_i$ for $i=1,2$ and $(b_i,c_i,l_i)\in
    \BCL_i(x,t)$ are all $C^1$ and $\partial \varphi (\ovs)$ is the convex hull of their gradients
    for all the $(b_i,c_i,l_i)$ such that $\varphi (\ovs)=-b_i\cdot(p'+se_N)+c_ir-l_i$.  Since
    $\BCL_1(x,t),\BCL_2(x,t)$ are convex, this means that one of the following cases holds
    \begin{enumerate}
        \item[$(a)$] either the above supremum is only achieved at a unique $(b_i,c_i,l_i)$ but then
        $\varphi$ is differentiable at $\ovs$ and $0=\varphi' (\ovs)=-b_i\cdot e_N$; 
        \item[$(b)$] or there exists $(b_1,c_1,l_1)\in \BCL_1(x,t)$, $(b_2,c_2,l_2)\in \BCL_2(x,t)$ and 
        $\mu \in [0,1]$ such that
        $$\begin{cases}\varphi (\ovs)=-b_1\cdot(p'+\ovs e_N)+c_1r-l_1=-b_2\cdot(p'+\ovs e_N)+c_2r-l_2\\
        0= \mu(-b_1\cdot e_N)+(1-\mu)(-b_2\cdot e_N)\quad\hbox{i.e.  }(\mu b_1+(1-\mu)b_2)\cdot
        e_N =0\; .\end{cases}$$
    \end{enumerate}

    In case $(b)$, we deduce that
    \begin{align}
    \varphi (\ovs) & = \mu(-b_1\cdot(p'+\ovs e_N)+c_1r-l_1)+(1-\mu)(-b_2\cdot(p'+\ovs e_N)+c_2r-l_2)\\
    & = -(\mu b_1+(1-\mu)b_2)\cdot p'+(\mu c_1+(1-\mu)c_2)r -(\mu l_1+(1-\mu)l_2)\\
    & \leq   \ \ \HT(x,t,r,p')\;.\end{align}
    But on the other hand, for any $(\tilde b_1, \tilde c_1,\tilde l_1) \in  \BCL_1(x,t)$, $(\tilde
    b_2,\tilde c_2,\tilde l_2)\in \BCL_2(x,t)$ such that $(\tilde\mu \tilde b_1+(1-\tilde \mu)\tilde
    b_2)\cdot e_N =0$ for some $\tilde\mu \in [0,1]$, the definition of $\varphi$ implies that
    \begin{align}
        \varphi (\ovs) & \geq  \tilde \mu(-\tilde b_1\cdot(p'+ \ovs e_N)+\tilde c_1 r-\tilde
        l_1)+(1-\tilde \mu)(-\tilde b_2\cdot(p'+\ovs e_N)+\tilde c_2 r-\tilde l_2)\\ & = -(\tilde\mu
        \tilde b_1+(1-\tilde\mu)\tilde b_2)\cdot p'+(\mu \tilde c_1+(1-\mu)\tilde c_2)r -(\mu \tilde
        l_1+(1-\mu)\tilde l_2),
    \end{align}
    which, taking the supremum on all such $(\tilde b_1, \tilde c_1,\tilde l_1)$, $(\tilde b_2,\tilde
    c_2,\tilde l_2)$ and $\tilde\mu$, gives $\varphi (\ovs)\geq \HT(x,t,r,p')$. Therefore, the
    equality holds, which gives the result.

    Dealing with case $(a)$ follows from the same arguments as in case $(b)$, with $\mu=0$ or $1$.
    Hence the Lemma is proved.  
\end{proof}

\section{A sufficient condition to get uniqueness}

Applying directly Proposition~\ref{prop:quasi.convex.maxreg} yields a condition under which
$\HT=\HT^\reg$.

\begin{lemma}\label{lem:H1m.H2p.c}
    We denote by $m_1^+(x,t,r,p')$ the largest minimum point of the function $s \mapsto H_{1}(x,t,r,p'+s
    e_N)$ and $m_2^-(x,t,r,p')$ the least minimum of the function $s \mapsto
    H_{2}(x,t,r,p'+s e_N)$.  If $m_1^+(x,t,r,p') \leq m_2^-(x,t,r,p')$ for any $(x,t,r,p')$ then
    $\HT=\HTreg$ on $\H \times [0,\Tf] \times \R \times \R^{N-1}$.
 \end{lemma}

The importance of this lemma is to give the
\begin{corollary}\label{cor:uniqueness.ishii}\emph{--- A uniqueness criterion for Ishii solutions.}\smsp 
    If $m_1^+(x,t,r,p') \leq m_2^-(x,t,r,p')$ for any $(x,t,r,p')\in \H \times [0,\Tf] \times \R
    \times \R^{N-1}$, there is a unique solution of \eqref{pb:half-space} in the sense of Ishii.
\end{corollary}

Therefore we have an easy-to-check sufficient condition in order to have $\VFm=\VFp$, \ie the
uniqueness of the Ishii solution. Moreover this condition can be checked directly on the
Hamiltonians $H_1,H_2$ without coming back to the control problem.

\begin{remark}\label{rem:upoH}
    In Part~\ref{part:NA}, we consider the more general case when $H_1,H_2$ are only quasi-convex. We
    point out that the above results, namely Lemma~\ref{lem:H1m.H2p.a} and \ref{lem:H1m.H2p.c} are
    of course still valid in the quasi-convex setting (in the codimension 1 case), provided that we
    use the definition of the $H_i^\pm$ through $f^\sharp,f_\flat, g^\sharp,g_\flat$. Indeed, in that
    way, the definitions do not require a control formulation. We come back later on this.  
\end{remark}

\section{More examples of uniqueness and non-uniqueness}\label{sec:more-examples}

In this section, we give two simple $1$-d examples to illustrate
Corollary~\ref{cor:uniqueness.ishii}. The first one is 
$$\begin{cases}
    u_t + |u_x-1|=0 \quad\hbox{in }(-\infty,0)\times (0,+\infty)\; ,\\
    u_t +  |u_x+1|=0 \quad\hbox{in }(0,+\infty)\times (0,+\infty)\; ,\\
    u(x,0)=|x| \quad\hbox{ in $\R$}\; .
\end{cases}$$
In this case, $m_1^+(x,t,r,p')= -1< m_2^-(x,t,r,p')=1$, uniqueness occurs and it is easy to compute
the value functions $$\VFm (x,t)=\VFp(x,t)=2(|x|-t)_+ - |x|-(t-|x|)_+=\begin{cases}
2(|x|-t)_+-|x| & \hbox{if $|x|\geq t$}\; ,\\
-t & \hbox{otherwise.}
\end{cases}$$

Next, consider the problem
$$\begin{cases}
    u_t + |u_x+1|=0 \quad\hbox{in }(-\infty,0)\times (0,+\infty)\; ,\\
    u_t +  |u_x-1|=0 \quad(0,+\infty)\times (0,+\infty)\; ,\\
    u(x,0)=|x| \quad\hbox{ in $\R$}\;.
\end{cases}$$
Here, on the contrary, $m_1^+(x,t,r,p')= 1> m_2^-(x,t,r,p')=-1$,
Corollary~\ref{cor:uniqueness.ishii} does not apply and actually the value functions are different
$$\VFm (x,t)=
\begin{cases} 
|x| & \hbox{if $|x|\geq t$} \\
2|x|-t  & \hbox{if $|x|\leq t$} \end{cases}
$$
while $\VFp (x,t)=|x| \;.$

\chapter{Adding a Specific Problem on the Interface}
\label{sec:H0.case}
\abstract{In this chapter, HJB Equations with an additional conditions on the hyperplane are
considered; these additional conditions correspond to a specific control problem on the interface.
We investigate the control formulas for the minimal and maximal solutions in this context.} 

This chapter is devoted to explain the main adaptations and differences when we consider the more
general problem 
\begin{equation}\label{pb:half-space.with.H0}
	\begin{cases}
	u_t+H_1(x,t,u,Du)=0 & \text{ for }x\in\Omega_1\;,\\
	u_t+H_2(x,t,u,Du)=0 & \text{ for }x\in\Omega_2\;,\\	
	u_t+H_0(x,t,u,D_T u)=0 & \text{ for }x\in\H\;,\\
    u(x,0)=\u0 (x) & \text{ for }x\in\R^N\;.
	\end{cases}
\end{equation}
Here, since $H_0$ is only defined on $\H$, the gradient $D_T u$ consists only on the tangential
derivative of $u$ if $x=(x',x_N)\in\R^{N-1}\times\R$, $D_T u = D_{x'}u$ (or $(D_{x'}u,0)$ depending
on the convention we choose). In order to simplify some formula, we may write $Du$ instead of $D_T
u$ and therefore $H_0(x,t,u,Du)$ instead of $H_0(x,t,u,D_T u)$, keeping in mind that $H_0$ depends
only on $p=Du$ through $p_T=D_T u$.

As we explained in Section~\ref{sect:stab}, the conditions on $\H$ for those equations have to be
understood in the relaxed (Ishii) sense, namely for \eqref{pb:half-space.with.H0}
\begin{equation}\label{eq:ishii.cond.H0}
\!\!\begin{cases}
	\max\Big(u_t+H_0(x,t,u,D_Tu),u_t+H_1(x,t,u,Du),u_t+H_2(x,t,u,Du)\Big)\geq0 \;,\\
	\min\Big(u_t+H_0(x,t,u,D_Tu),u_t+H_1(x,t,u,Du),u_t+H_2(x,t,u,Du)\Big)\leq0 \;,\\
	\end{cases}
\end{equation}
meaning that, for the supersolution \resp{subsolution} condition, at least one of the inequations
has to hold. 

In this section, we use the notation with $H_0$ as a sub/superscript in the mathematical objects to
differentiate from the ``non''-$H_0$ case since these are not exactly the same, in particular of
course, the value functions differ whether we have a specific control problem on $\H$ or not.

We say here that the ``standard assumptions in the codimension-$1$ case'' are satisfied for
\eqref{pb:half-space.with.H0} if \HBACP holds for $(b_i,c_i,l_i)$, $i=0,1,2$ and \NCoH holds for
$H_1$ and $H_2$.

\section{The control problem}

The control problem is solved exactly as in the case of \eqref{pb:half-space} that was considered
above. We just need to add a specific control set $A_0$ and triples $(b_0,c_0,l_0)$, defining
$\BCL_0(x,t)$ when $x\in\H$ as for $\BCL_1$ and $\BCL_2$.  Since the case $i=0$ is specific because
$\H$ can be identified with $\R^{N-1}\times\{0\}$, we set for all $(x,t,\alpha_0)$,
$b_0(x,t,\alpha_0)=(b_0'(x,t,\alpha_0),0)$ so that $b_0\cdot p$ reduces to the scalar product of the
first $(N-1)$ components.

Using this convention, we define now the new $\BCL$ as
$$ \BCLHO(x,t):=
\begin{cases}
	\BCL_1(x,t) & \text{if }x\in \Omega_1\;,\\
    \BCL_2(x,t) & \text{if }x\in\Omega_2\;,\\
    \cob(\BCL_0,\BCL_1,\BCL_2)(x,t) & \text{if }x\in\H\;,
\end{cases}
$$
where the convex hull takes into account here the three sets $\BCL_i$ for $i=0,1,2$ so that of
course, on $\H$ we make a convex combination of all the $(b_i,c_i,l_i)$, $i=0,1,2$. 
\begin{lemma}
	The set-valued map $\BCLHO$ satisfies \HBCL.
\end{lemma}
The proof is an obvious adaptation of Lemma~\ref{lem:HBCL}, therefore we skip it. 

In order to describe the trajectories of the differential inclusion with $\BCLHO$, we have to
enlarge the control space with $A_0$ (and introduce a new parameter $\mu_0$ for the convex
combination)
$$\AHO:= A_0\times A_1 \times A_2\times\tilde\Delta\;,\quad\text{and}\quad 
\mAHO:=L^\infty(0,\Tf;\AHO)\;.$$
Here, $\tilde\Delta=\{(\mu_0,\mu_1,\mu_2)\in[0,1]^3: \mu_0+\mu_1+\mu_2=1\}$, so that the extended
control takes the form $a=(\alpha_0,\alpha_1,\alpha_2,\mu_0,\mu_1,\mu_2)$ and if $x\in\H$,
$$(b_\H,c_\H,l_\H)=\mu_0(b_0,c_0,l_0)+\mu_1(b_1,c_1,l_1)+\mu_2(b_2,c_2,l_2)\, ,$$ with $\mu_0+\mu_1
+\mu_2 =1$.

With this modification, solving the differential inclusion with $\BCLHO$ and the description of
trajectories is similar to that in the $\BCL$-case (see Lemma~\ref{lem:struc.traj}), except that the
control has the form $a(\cdot)=(\alpha_0,\alpha_1,\alpha_2,\mu_0,\mu_1,\mu_2)(\cdot)\in\mAHO$.

Then we define $\VFmHO$ by 
$$\VFmHO(x,t):=\inf_{\mTHO(x,t)}\left\{\int_0^tl(X(s),t-s,a(s))\exp(-D(s))\ds+\u0
(X(t))\exp(-D(t))\right\}\;,$$ 
where $\mTHO(x,t)$ is the space of trajectories associated with $\BCLHO$.

\section{The minimal solution}

As far as the value function $\VFmHO$ is concerned, only easy adaptations are needed to handle $H_0$
and the related control problem.  Of course we assume that $H_0$ also satisfies \HConv, \NCe, \TC and
\Monu, as it is the case for $H_1$ and $H_2$.

Lemma \ref{lem:global.super} holds here with 
$$\begin{aligned}
 \HHO (x,t,u,p) &:= \sup_{(b,c,l)\in\BCLHO (x,t)}\big(-b\cdot p+cu-l\big)\;,\\
 \FHO (x,t,u,(p_x,p_t)) &:= p_t+ \HHO(x,t,u,p)\;,
\end{aligned}$$
and of course we have to add $H_0$ in the $\max$ of the right-hand sides
$$\HHO(x,t,r,p)=\max\Big(H_0(x,t,r,p),H_1(x,t,r,p),H_2(x,t,u,p)\Big) \;,$$
$$
    \FHO(x,t,u,(p_x,p_t))=\max \big(p_t+H_0(x,t,r,p),p_t+ H_1(x,t,u,p),p_t+ H_2(x,t,u,p)\big)\;. 
$$
Then, minimality of $\VFmHO$ follows exactly as in Proposition~\ref{prop:ishii}
\begin{proposition}\label{prop:ishii.HO} 
    Assume that the ``standard assumptions in the codimension-$1$ case'' are satisfied for
    \eqref{pb:half-space.with.H0}. Then the value function $\VFmHO$ is an Ishii viscosity solution
    of \eqref{pb:half-space.with.H0}. Moreover $\VFmHO$ is the minimal supersolution of
    \eqref{pb:half-space.with.H0}.
\end{proposition}

Notice that a tangential dynamic $b\in\B_T^{H_0}(x,t)$ is expressed as a convex combination
\begin{equation}
    b=\mu_0 b_0+\mu_1 b_1+\mu_2 b_2
\end{equation}
for which $\mu_0+\mu_1+\mu_2=1$, $\mu_0,\mu_1,\mu_2\in [0,1]$ and $(\mu_1 b_1+\mu_2 b_2)\cdot e_N=0$
since, here, by definition, $b_0\cdot e_N=0$.

Then, all the results of Section~\ref{subsec:complemented} apply, except that we need a little
adaptation for Lemma~\ref{bon-bcl} in order to take into account the $b_0$-contribution.

\begin{proof}[Proof of Lemma~\ref{bon-bcl} in the $\BCLHO$-case]
    The only modification consists in rewriting the convex combination as
    $$
        \mu_0 b_0(x,t,\alpha_0)+(1-\mu_0)\left(\frac{\mu_1}{1-\mu_0} b_1(x,t,\alpha_1)+
        \frac{\mu_2}{1-\mu_0} b_2(x,t,\alpha_2)\right)\;,
    $$
    and we apply the arguments of Lemma~\ref{bon-bcl} to the convex combination
    $$ 
        \frac{\mu_1}{1-\mu_0} b_1(x,t,\alpha_1)+\frac{\mu_2}{1-\mu_0} b_2(x,t,\alpha_2)\;.
    $$
    Then, setting
    $$ 
        \psi^{H_0}(y,s):= \mu_0 b_0(x,t,\alpha_0)+(1-\mu_0)\left(\mu^\sharp_1(y,s)(b_1,c_1,l_1) 
        + \mu^\sharp_2(y,s)(b_2,c_2,l_2)\right)\;,
    $$
    it is easy to check that the lemma holds for the $\BCLHO$-case.
\end{proof}

Finally, the minimal solution $\VFmHO$ can also be characterized through $\HHO_T$. The proof follows
exactly the ``non-$H_0$'' case with obvious adaptations so that we omit it.
\begin{theorem}\label{thm:minimal.charac.HO}
    Assume that the ``standard assumptions in the codimension-$1$ case'' are satisfied for
    \eqref{pb:half-space.with.H0}. Then $\VFmHO$ is the unique Ishii solution of
    \eqref{pb:half-space.with.H0} such that 
    $$u_t+\HHO_T(x,t,u,D_Tu)\leq0\quad\text{on}\quad\H\times (0,\Tf)\;,$$
    where, for $x\in \H$, $t\in [0,\Tf]$, $r\in \R$, $p\in \R^{N-1}$,
    $$\HHO_T (x,t,r,p) := \sup_{(b,c,l)\in\BCLHO_T (x,t)}\big(-b\cdot p+cu-l\big)\;,$$
    $\BCLHO_T (x,t)$ being the subset of all $(b,c,l)\in \BCLHO(x,t)$ for which $b \in
    \B_T^{H_0}(x,t)$. 
\end{theorem}

\section{The maximal solution}

Surprisingly, for the maximal solution, the case of \eqref{pb:half-space.with.H0} is very different.
And we can see it on the result for subsolutions, analogue to Lemma~\ref{subsol-H}
\begin{lemma}\label{subsol-H0}
    If $u:\R^N\times (0,\Tf)\to \R$ is an \usc subsolution of \eqref{pb:half-space}, then it satisfies
    \begin{equation}\label{uH62}
        u_t+\min\big(H_0(x,t,u,D_T u), \HTreg (x,t,u,D_T u)\big)\leq 0\quad \hbox{on  }\H\times
        (0,\Tf)\;.  
    \end{equation}
\end{lemma}

We omit the proof since it is the same as that of Lemma~\ref{subsol-H} (taking into account the
$b_0$-terms), but of course the conclusion is that the $H_0$-inequality necessarily holds if the
$\HTreg$ does not, hence the min.

The important fact in Lemma~\ref{subsol-H0} is that, while, without $H_0$, \eqref{uH61} keeps the
form of an HJB-inequality for a control problem, it is not the case anymore for \eqref{uH62} where
the $\min$ looks more like an Isaacs equation associated to a differential game. As we already
mention it in the introduction of this part, this is the analogue for discontinuities of the
phenomena which arises in exit time problems/Dirichlet problem where the maximal Ishii subsolution
involves a ``worse stopping time'' on the boundary: we refer to \cite{BP2} and \cite{Ba} for
details.

As an illustration, let us provide the form of the maximal solution of \eqref{pb:half-space.with.H0}
in the particular case when for any $x\in \H$, $t\in (0,\Tf)$, $r\in \R$ and $p_T \in \R^{N-1}$
\begin{equation}\label{HTreggeqH0}
    H_0(x,t,r,p_T )\leq \HTreg (x,t,u,p_T)\; .
\end{equation}

\begin{proposition}\label{soussolmaxII} 
    Assume that the ``standard assumptions in the codimension-$1$ case'' are satisfied and assume
    that \eqref{HTreggeqH0} holds. Let $V:\H\times (0,\Tf) \to \R$ be the unique solution of 
    $$ u_t+H_0(x,t,u,D_T u) = 0\quad \hbox{on  }\H\times (0,\Tf)\;,$$
    with the initial data $(u_0)_{|\H}$. For $i=1,2$, let $V_i : \Omega_i \times [0,\Tf]\to \R$ be
    the unique solutions of the problems 
    $$\begin{cases} u_t+H_i(x,t,u,D u)  =   0\quad & \hbox{on  }\Omega_i \times (0,\Tf)\; ,\\
        u(x,t)  =  V(x,t) \quad & \hbox{on  }\H\times (0,\Tf)\; ,\\
        u(x,0)  =  (u_0)_{|\Omegb_i} \quad & \hbox{on  }\Omegb_i\; .
    \end{cases}$$
    Then the maximal (sub)solution of \eqref{pb:half-space.with.H0} is given by
    $$ \VFpHO(x,t)=
        \begin{cases}
        V_i(x,t) & \hbox{if $x \in \Omega_i$} \\
        V(x,t) & \hbox{if $x \in \H$}\; .
        \end{cases}
    $$
\end{proposition}

Before giving the short proof of Proposition~\ref{soussolmaxII}, we examine a simple example in
dimension $1$ showing the main features of this result. We take 
$$ \BCL_1(x,t):=\{(\alpha,0,0);\ |\alpha|\leq 1\}\; ,$$
$$ \BCL_2(x,t):=\{(\alpha,0,1);\ |\alpha|\leq 1\}\; ,$$
and $\BCL_0(0,t)=\{(0,0,2)\}$. In which case
$$ H_1(p)=|p|\; ,\;H_2(p):=|p|-1\; ,\; \HTreg = 0\; ,\; H_0 = -2\; .$$
Hence \eqref{HTreggeqH0} holds. It is easy to check that, if $u_0(x) =0$ for all $x \in \R$
$$ V(t)=2t\; ,\; V_1(x,t)=0\; ,\; V_2(x,t)=t\quad \hbox{for  }x\in \R,\ t\geq 0 .$$
This example shows several things: first, the value function $\VFpHO$ is discontinuous although we
have controllability/coercivity for the Hamiltonians $H_1$ and $H_2$; it is worth pointing out
anyway that the global coercivity is lost since we use the Hamiltonian $\min(H_0,H_1,H_2)$ on $\H$
for the subsolutions instead of $\min(H_1,H_2)$. 

Then, the values of $V(t)$ may seem strange since we use the maximal cost $2$ but as we mention it
above, this phenomena looks like the ``worse stopping time'' appearing in exit time problems.
Finally, and this is even more surprising, the form of $\VFpHO$ shows that no information is
transfered from $\Omega_1$ to $\Omega_2$: indeed, from the control point of view, starting from
$x<0$ where the cost is $1$, it would seem natural to cross the border $0$ to take advantage  of the
$0$-cost in $\Omega_1$ but this is not the case, even if $x<0$ is close to $0$. We have here two
state-constrained problems, both in $\Omega_1 \times [0,\Tf]$ and $\Omega_2 \times [0,\Tf]$. This also
means that the differential games features not only implies that one is obliged to take the maximal
cost at $x=0$ but also may prevent the trajectory to go from a less favourable region to a more
favourable region.

Unfortunately we are unable to provide a general formula for $\VFpHO$, \ie which would be valid for
all cases without \eqref{HTreggeqH0}. Of course, trying to define $\VFpHO$ as in
Proposition~\ref{soussolmaxII} but $V$ being the solution of
\begin{equation}\label{sol-min-H0}
    u_t+\min\Big\{H_0(x,t,u,D_T u), \HTreg (x,t,u,D_T u)\Big\}= 0\quad \hbox{on  }\H\times
    (0,\Tf)\;, 
\end{equation}
does not work as the following example shows. In dimension $1$, we take $H_1(p)=H_2(p)=|p|$, $H_0>0$
and $u_0(x) =-|x|$ in $\R$.  Since $\HTreg = 0$, we have $H_0>\HTreg$ and solving the above pde
gives $V = 0$. Computing $V_1$ and $V_2$ as above gives $-|x|-t$ in both cases. Hence $V_1$ and $V_2$
are just the restriction to $\Omega_1 \times [0,\Tf]$ and $\Omega_2 \times [0,\Tf]$ respectively of
the solution of
$$ u_t + |u_x|=0\quad \hbox{in  }\R\times (0,\Tf)\; ,$$
with the initial data $u_0$. Now defining $\VFpHO$ as in Proposition~\ref{soussolmaxII}, we see that
we do not have a subsolution: indeed the discontinuity of $\VFpHO$ at any point $(0,t)$ implies that
$(0,t)$ is a maximum point of $\VFpHO - px$ for any $p\in \R$ and therefore we should have the
inequality
$$\min( H_0, |p|, |p|) \leq 0\; ,$$
which is not the case if $|p|>0$. 

\begin{remark} 
    Even if we were are able to provide a general formula for $\VFpHO$, we have some (again strange)
    information on this maximal subsolution: first $\VFpHO \geq \VFp$ in $\R^N \times (0,\Tf)$
    since $\VFp$ is a subsolution of \eqref{pb:half-space.with.H0}. A surprising result since it
    shows that adding $H_0$ on $\H\times (0,\Tf)$ does not decrease the maximal subsolution as it
    could be thought from the control interpretation. On the other hand, Lemma~\ref{subsol-H0}
    provides an upper estimate of $\VFpHO$ on $\H\times (0,\Tf)$, namely the solution of
    \eqref{sol-min-H0}.
\end{remark}

\begin{proof}[Proof of Proposition~\ref{soussolmaxII}] 
    First, by our assumptions, $V$ exists and is continuous, since it is obtained by solving
    a standard Cauchy problem in $\R^{N-1} \times [0,\Tf]$. Next by combining the argument of
    \cite{BP2} (See also \cite{Ba}) with the localization arguments of  Section~\ref{sect:htc},
    $V_1$ and $V_2$ exist and are continuous in $\Omega_1\times[0,\Tf]$ and $\Omega_2\times[0,\Tf]$
    respectively, with continuous extensions to $\Omegb_1\times[0,\Tf]$ and
    $\Omegb_2 \times[0,\Tf]$. 

    Considering the Cauchy-Dirichlet problems in $\Omega_1$ and $\Omega_2$, we refer the reader to
    Proposition~\ref{Dir-MN}-$(i)$ where it is proved that the normal controllability implies
    $$
        V_1(x,t), V_2(x,t) \leq V(x,t) \quad \hbox{on  }\H\times (0,\Tf)\; .
    $$
    Hence, defined in that way, $ \VFpHO$ is upper semicontinuous (it may be discontinuous as we
    already saw above).

    It is easy to check that $ \VFp$ is a solution of \eqref{pb:half-space.with.H0}. Indeed the
    subsolution properties on $\Omega_1\times (0,\Tf),\Omega_2\times (0,\Tf)$ are obvious. On
    $\H\times (0,\Tf)$ they come from the properties of $V$ since $ \VFpHO=V$ on $\H\times (0,\Tf)$;
    hence the $H_0$-inequality for $V$ implies the subsolution inequality for $ \VFp$.

    For the supersolution ones, they comes from the properties of $V_1$, $V_2$ and $V$ and the
    formulation of the Dirichlet problem since $(\VFpHO)_*=\min(V_1,V_2,V)=\min(V_1,V_2)$ on
    $\H\times (0,\Tf)$. Indeed if $\phi$ is a smooth function in $\R^N \times (0,\Tf)$ and if
    $(\xb,\tb)\in \H\times (0,\Tf)$ is a minimum point of $(\VFpHO)_*-\phi$, there are several
    cases:

    \noindent $(a)$ if $(\VFpHO)_*(\xb,\tb)=V_1(\xb,\tb)< V(\xb,\tb)$, then $(\xb,\tb)$ is a minimum
    point of $V_1-\phi$ on $\Omegb_1 \times (0,\Tf)$ and, since $V_1$ is a solution of
    the Dirichlet problem in $\Omegb_1 \times (0,\Tf)$ with the Dirichlet data $V$, we
    have
    $$ \max \big(\phi_t (\xb,\tb)+H_1(\xb,\tb,V_1(\xb,\tb),D \phi(\xb,\tb)),V_1(\xb,\tb)-
    V(\xb,\tb)\big )\geq 0\;.$$
    Hence $\phi_t (\xb,\tb)+H_1(\xb,\tb,V_1(\xb,\tb),D \phi(\xb,\tb))\geq 0$, which gives the answer
    we wish. 

    \noindent $(b)$ The case when $(\VFpHO)_*(\xb,\tb)=V_2(\xb,\tb)< V(\xb,\tb)$ is treated in a
    similar way. 

    \noindent $(c)$ Finally if $(\VFpHO)_*(\xb,\tb)=V_1(\xb,\tb)=V_2(\xb,\tb) = V(\xb,\tb)$, we use
    that $(\xb,\tb)$ is a minimum point of $V-\phi$ on $\H\times (0,\Tf)$ and therefore
    $$ \phi_t (\xb,\tb)+H_0(\xb,\tb,V(\xb,\tb),D \phi(\xb,\tb))\geq 0\; ,$$ 
    implying the viscosity supersolution inequality we wanted.

    It remains to prove that any subsolution $u$ of \eqref{pb:half-space.with.H0} is below $
    \VFpHO$. This comes from Lemma~\ref{subsol-H0} which implies, using a standard comparison result
    on $\H\times [0,\Tf]$ that $u(x,t) \leq  V(x,t) = \VFpHO(x,t)  $ on $\H\times [0,\Tf]$.  
\end{proof}

\chapter{Remarks on the Uniqueness Proofs, Problems Without Controllability}
\label{chap:rem.uniq.cont}

\abstract{The aim of this short chapter is twofold: first, it analyzes the uniqueness proof and then
considers cases where the ``good assumptions'' are not satisfied; in particular when the normal
controllability does not hold.}

\section{The main steps of the uniqueness proofs and the role of the normal controllability}
\label{sec:MSUP}

In this part, we have proved several comparison results showing, on one hand, that $\VFm$ is the
minimal supersolution and the unique solution which satisfies the $\HT$-inequality and, on the other
hand, that $\VFp$ is the maximal subsolution and the unique solution which satisfies the
$\HTreg$-inequality.

All the proofs of these results are based on a common strategy which will also be used for
stratified problems in Part~\ref{stratRN} and which can be described in the following ``backwards''
way 
\begin{enumerate}[leftmargin=4em]
    \item[{\bf Step 3 :}] \textbf{The ``Magical Lemma''.} \index{Magical Lemma!role in the comparison proof}
According to Section~\ref{sect:htc} the comparison result
        is reduced to proving that \LCR holds. For the points located on $\H$, this is a direct
        consequence of Lemma~\ref{lem:comp.fundamental} if the subsolution is continuous and $C^1$
        in the tangential variables. This tangential regularity allows to use the subsolution as a
        test-function for the ``tangential inequalities'' (typically the $\HT$ or $\HTreg$ one),
        avoiding in particular the usual ``doubling of variables'' which causes the major problem in
        the discontinuous setting.
    \item[{\bf Step 2:}]\textbf{Regularization of the subsolution.} In order to use the ``Magical Lemma'' to
        obtain the result for any subsolution, we have to be able to regularize any subsolution in order
        that it becomes continuous w.r.t. all the variables, $C^1$ in the tangent variables, and
        preserving the subsolution inequalities. This is the role of Propositions~\ref{reg-by-sc}
        and~\ref{C1-reg-by-sc}.
    \item[{\bf Step 1:}]\textbf{Regularity of the subsolution.} In order to perform the second step in a suitable
        way, we need at least the subsolution to be regular on $\H$. In particular this is necessary
        in order that the second step actually provides a subsolution which is continuous on $\H$
        (but also on the hyperplanes which are parallel to $\H$).
\end{enumerate}

Going further in the analysis of these three steps, it is clear that the normal controllability
assumption \NCe plays a crucial role in Step 1 but even more in Step 2. Looking at
Proposition~\ref{reg-sub}, recalling that \NCe implies \NCw, Case (a) immediately gives us the
complete information we need, even if we can obtain it through Cases (b) and (c) in some situations,
see the examples below.

But this is in Step 2 that \NCe plays the most important (an maybe unavoidable) role: in
order to perform the tangential regularization we have to control, one way or the other, the normal
component of the gradient. This is exactly the role of \NCe. 

This is why we consider \NCe as a key ``natural'' assumption in this type of problems and the fact
that the same remarks can be made for stratified problems reinforces this certainty. Being unable to
perform the regularization process, the ``Magical Lemma'' cannot be used and all the proofs
collapse.

We also point out that the approach via ``Flux-Limited  Solutions'' described
in Part~\ref{part:NA} provides an alternative strategy which seems to avoid some of the above
constraints, and in particular \NCe. The comparison proof is based on an ``almost classical''
doubling of variables but the reader can check that this proof actually uses \NCe in several ways.

However, some problems without normal controllability can also be treated and we give some examples
in the next section.

\section{Some problems without controllability}

In this section, we are not going to examine sophisticated situations: if \NCe is not
satisfied on $\H$ and if we have a mixture of the different ``simple'' situations we describe below
on $\H$, we are led to problems whose difficulties have to be examined separately. A combination of
the arguments which are presented in this book may allow to treat such problems but, in a general
framework, this will not be the case.

In the simple situations we are going to emphasize, we examine the situation separately on both
sides of the discontinuity and we respectively denote by $u_1$ and $u_2$ the solutions in
$\Omegb_1 \times (0,\Tf)$ and $\Omegb_2 \times (0,\Tf)$. Therefore the value
function $\VF$ of the control problem in $\R^N\times (0,\Tf)$ will be given by
$$
    \VF(x,t)=\begin{cases} u_1(x,t) & \hbox{if }x\in \Omega_1\; ,\\
    u_2(x,t) & \hbox{if }x\in \Omega_2\;\end{cases}
$$
while on $\H \times (0,\Tf)$, either $\VF$ will be the common value of $u_1$ and $u_2$ or the
\lsc/\usc envelopes, computed by using values in $\Omega_1 \times (0,\Tf)$ and $\Omega_2 \times
(0,\Tf)$.

The simple situations we have in mind are the following
\begin{enumerate}
    \item[{\bf I.}] For any $x\in \H$, $t\in [0,\Tf]$, $\alpha_2 \in A_2$, $b_2(x,t,\alpha_2)\cdot
        e_N >0$. All the dynamics used in $\Omega_2$ are strictly pointing outside
        $\Omegb_2$ on $\H$. Here, it is easy to show that $H_2$ plays the role of a
        nonlinear Neumann boundary condition on $\H$ for the equation $H_1=0$ in $\Omega_1\times
        (0,\Tf)$\footnote{We give a proof at the end of this section for the reader's convenience}.
        Therefore, in order to obtain $u_1$, we solve this nonlinear Neumann  problem in
        $\Omega_1\times (0,\Tf)$. We obtain a unique continuous solution $u_1$ (which is continuous up
        to the boundary). Then, in order to compute $u_2$, we solve the Dirichlet problem in
        $\Omegb_2\times (0,\Tf)$ with $u_1$ as Dirichlet boundary condition on $\H \times
        (0,\Tf)$. This also provides a continuous solution in $ \Omegb_2\times (0,\Tf)$
        and that way, we have defined a continuous function in $\R^N$ which is the solution of
        Problem~\ref{pb:half-space} and the value function of the associated control problem.

    \item[{\bf II.}] By symmetry the situation is the same if, for any $x\in \H$, $t\in [0,\Tf]$,
        $\alpha_1 \in A_1$, $b_1(x,t,\alpha_1)\cdot e_N<0$.

    \item[{\bf III.}] For any $x\in \H$, $t\in [0,\Tf]$, $\alpha_2 \in A_2$, $b_2(x,t,\alpha_2)\cdot
        e_N \leq 0$. Then all the trajectories of the dynamic starting in $\Omega_2\times (0,\Tf)$
        stay in $\Omega_2\times (0,\Tf)$. In terms of PDE, the consequence is that all the
        viscosity inequalities for sub and supersolutions hold up to the boundary of $\Omega_2\times
        (0,\Tf)$, as soon as these sub and supersolutions are extended up the boundary by upper or
        lower-semicontinuity. Hence the associated HJB problem is $H_2=0$ on
        $\Omegb_2\times (0,\Tf)$. As in the whole space $\R^N$, this problem enjoys a
        comparison result and therefore it provides a unique solution $u_2\in C(\Omegb_2\times
        (0,\Tf))$. This solution is the value function in $\Omega_2\times \Tf$, extended to
        $\Omegb_2\times (0,\Tf)$ by continuity\footnote{Therefore $u_2$ is equal to the
        $\R^N$-value function in $\Omega_2\times \Tf$ but maybe not on $\H\times (0,\Tf)$.}.
        Therefore the problem in $\Omega_2$ completely ignores the problem in $\Omega_1$, and we
        face 3 different cases for the problem in $\Omega_1$
    \begin{enumerate}
     \item[{\bf III.1}] For any $x\in \H$, $t\in [0,\Tf]$, $\alpha_1 \in A_1$,
            $b_1(x,t,\alpha_1)\cdot e_N<0$, a case which is already treated in {\bf II} above. 
            But here we are in the case of a simple Dirichlet problem in $\Omega_1\times \Tf$, the
            Dirichlet boundary condition on $\H\times \Tf$ being the value function of the problem
            in $\Omega_2$. Hence there is a unique continuous solution for
            Problem~\ref{pb:half-space} which is the value function of the control problem in
            $\R^N$. 

        \item[{\bf III.2}] For any $x\in \H$, $t\in [0,\Tf]$, there exist $\alpha_1^1, \alpha_1^2
            \in A_1$ such that $b_1(x,t,\alpha_1^1)\cdot e_N>0$ and $b_1(x,t,\alpha_1^2)\cdot
            e_N<0$, $i.e.$ the normal controllability condition holds. In this case, we also have a
            Dirichlet problem in $\Omega_1\times \Tf$ with the Dirichlet boundary condition on $\H$
            being the value function of the problem in $\Omega_2\times \Tf$. However, while in {\bf
            III.1} the boundary data is assumed in a classical sense and leads to a continuous
            solution in $\R^N$, here it is only assumed in the viscosity sense. The value function
            in $\Omega_1\times \Tf $ being not equal, in general, to the one in $\Omega_2\times \Tf$
            in all $\H\times \Tf$, the value function of the problem in $\R^N$ may have
            discontinuities on $\H\times \Tf$.  

        \item[{\bf III.3}] For any $x\in \H$, $t\in [0,\Tf]$, $\alpha_1 \in A_1$,
            $b_1(x,t,\alpha_1)\cdot e_N\geq 0$: here the problem in $\Omega_1\times \Tf$ and
            $\Omega_2\times \Tf$ are completely independent. There exists both a unique value
            function in $\Omegb_1\times \Tf$ and $\Omegb_2\times \Tf$ but
            their continuous extensions to $\H\times \Tf$ are different in general, and the value
            function in $\R^N$ may have discontinuities on $\H\times \Tf$. Anyway the Ishii
            conditions are satisfied on $\H\times \Tf$ since both equations hold up to the boundary.
\end{enumerate}
\end{enumerate}

\

We conclude this section by proving that, as announced in Case~{\bf I} above, dynamics pointing
outward generate a nonlinear Neumann boundary condition.  
\begin{proposition}\label{prop:dyn.neumann}
    Assume that the ``standard assumptions in the codimension-$1$ case'' are satisfied and that,
    for any $x\in \H$, $t\in [0,\Tf]$, $\alpha_2 \in A_2$, $b_2(x,t,\alpha_2)\cdot e_N >0$. Then any
    locally bounded \usc subsolution \resp{ \lsc supersolution $v$} of Problem~\ref{pb:half-space}
    is a subsolution \resp{supersolution} of the nonlinear Neumann problem
    \begin{equation}
        \left\{ \begin{array}{ccc}u_t+H_1(x,t,u,D_x u) & = & 0\quad \hbox{in  }\Omega_1 \times (0,\Tf)\\
        u_t+H_2(x,t,u,D_x u) & = & 0\quad \hbox{on  }\H \times (0,\Tf)\; .
        \end{array}\right.
    \end{equation}
\end{proposition}

We recall that Neumann boundary conditions for first-order HJ Equations were first studied by Lions
\cite{Li-Neu} and then different comparison results for first and second-order equations were
obtained by Ishii \cite{Is-Neu} and Barles\cite{Ba-Neu}. We refer the reader to the ``User's guide
to viscosity solutions'' of Crandall, Ishii and Lions \cite{Users} for a complete introduction of
boundary conditions in the viscosity sense and to all these references for checking that the
nonlinearity $p_t+H_2(x,t,r,p_x)=0$ satisfies all the requirement for a nonlinear Neumann boundary
condition.

\begin{proof}
    Of course, we just have to check the boundary condition and we provide the proof only in the
    subsolution case, the supersolution one being analogous. Let $\phi \in C^1(\R^N \times (0,\Tf))$
    and let $(\xb,\tb) \in \H \times (0,\Tf)$ be a strict local maximum point of $u-\phi$. For
    $0<\e\ll 1$, we consider the penalized function
    $$ 
        (x,t) \mapsto u(x,t)-\phi(x,t)-\frac{[(x_N)_-]^2}{\e}
    $$
    An easy application of Lemma~\ref{lem:cv-pen} in a compact neighborhood of $(\xb,\tb)$ shows the
    existence of a sequence $(\xe,\te)$ of maximum points for these functions such that
    $(\xe,\te)\to (\xb,\tb)$ and $u(\xe,\te)\to u(\xb,\tb)$. If $(\xe,\te)\in
    \Omegb_1\times \Tf$, we have
    $$\begin{aligned}\text{either } & 
        \phi_t (\xe,\te)+H_1(\xe,\te,u(\xe,\te),D_x \phi(\xe,\te)) \leq 0\\
        \text{or } & \phi_t (\xe,\te)+H_2(\xe,\te,u(\xe,\te),D_x \phi(\xe,\te)) \leq 0\;,
    \end{aligned}$$
    the derivative of the term $\dfrac{[(x_N)_-]^2}{\e}$ being $0$. Hence the only difficulty is when
    $(\xe,\te) \in \Omega_2\times \Tf$ and
    $$ 
    \phi_t \Big(\xe,\te)+H_2(\xe,\te,u(\xe,\te),D_x \phi(\xe,\te)- \frac{2(x_N)_-}{\e}e_N\Big)\leq 0\;.
    $$
    But examining $H_2$ and using the fact that, for any $x\in \H$, $t\in [0,\Tf]$, $\alpha_2 \in
    A_2$, $b_2(x,t,\alpha_2)\cdot e_N >0$, we see that $\lambda\mapsto H_2(x,t,r,p_x +\lambda e_N)$
    is decreasing for all $x\in \H$, $t\in [0,\Tf]$, $r\in \R$ and $p_x \in \R^N$. Therefore 
    $$
    H_2(\xe,\te,u(\xe,\te),D_x \phi(\xe,\te)- \frac{2(x_N)_-}{\e}e_N) \geq H_2(\xe,\te,u(\xe,\te),D_x
    \phi(\xe,\te))
    $$
    and we also get in this case
    $$ \phi_t (\xe,\te)+H_2(\xe,\te,u(\xe,\te),D_x \phi(\xe,\te)) \leq 0\; .$$
    In any case
    $$\begin{aligned}
        \min\Big( & \phi_t (\xe,\te)+H_1(\xe,\te,u(\xe,\te),D_x \phi(\xe,\te)), \\
        & \phi_t (\xe,\te)+H_2(\xe,\te,u(\xe,\te),D_x \phi(\xe,\te))\Big) \leq 0\; ,
    \end{aligned}$$
    and letting $\e \to 0$, we obtain the desired inequality
    $$\min\Big(\phi_t (\xb,\tb)+H_1(\xb,\tb,u(\xb,\tb),D_x \phi(\xb,\tb))\ ,\  
    \phi_t (\xb,\tb)+H_2(\xb,\tb,u(\xb,\tb),D_x \phi(\xb,\tb)\Big) \leq 0\; .$$
\end{proof}

\chapter{Further Discussions and Open Problems}
\label{chap:COP-II}

\abstract{The discussion focuses on the Ishii subsolution inequality and extensions to stationary
problems; then, more general discontinuities are presented, leading to puzzling open problems.}

\section{The Ishii subsolution inequality: natural or unnatural from the control point of view?}
\label{sect:natural.Ishii}

As it is well-known, the Ishii supersolution inequality is very natural from the control point of
view, and even in a very general framework. The reader can be convinced by this claim by looking at
Chapter~\ref{chap:control.tools}, and in particular at Theorem~\ref{SP} and
Corollary~\ref{VFm-minsup}: involving the natural $\F$-Hamiltonian, the proof that the value function is a
supersolution---and even the minimal supersolution---is rather easy and reflects as expected the
property of the control problem, since it is related to the existence of an optimal trajectory.

On the contrary, the proof of the subsolution inequality---which has to handle $\F_*$---is far more
involved, \cf Theorem~\ref{thm:SubP}, and no analogue of Corollary~\ref{VFm-minsup} exists.  This
rises the question: is this Ishii subsolution inequality so natural from the control point of view?

\

\noindent\textsc{Why the Ishii inequality should not hold ---}
We can provide the beginning of an answer in a rather simple way in the two-domains case. We recall
that the role of the subsolution inequality is to reflect the fact that each control (or trajectory)
is suboptimal.

If $U=\VFm$ or $\VFp$, if $(x,t)\in \H \times (0,\Tf)$ and if $\alpha_1$ is a control such that
$b_1(x,t,\alpha_1)\cdot e_N >0$, we solve the ode
$$ \dot X(s)=b_1(X(s),t-s,\alpha_1)\; ,\; X(0)=x\; ,$$ and we remark that, for $s>0$ small enough, $X(s)\in \Omega_1$. Therefore the trajectory $X(\cdot)$ is
admissible and an easy application of the Dynamic Programming Principle (where we assume that we
already know that $U$ is continuous for simplicity) implies, for $h>0$ small enough $$ U(x,t) \leq \int_{0}^{h}
l\big(X(s),t-s,\alpha_1\big)\,e^{-D(s)}\ds + U ( X(h), t-h) \,e^{-D(h)}\;.$$ We easily deduce that,
for such $\alpha_1$ \begin{equation}\label{eqn:ineq-sub-a1} -b_1(x,t,\alpha_1)\cdot DU(x,t) +
c_1(x,t,\alpha_1)U(x,t) -l_1(x,t,\alpha_1)\leq 0 \end{equation} and this inequality can easily be
extended to all $\alpha_1$ such that $b_1(x,t,\alpha_1)\cdot e_N \geq 0$.

On the contrary, if $b_1(x,t,\alpha_1)\cdot e_N < 0$, $X(s) \in \Omega_2$ for $s>0$ small enough and $X(\cdot)$
is not an admissible trajectory anymore since the dynamic is $b_2$ in $\Omega_2 \times (0,\Tf)$; so there is no reason why \eqref{eqn:ineq-sub-a1}
should hold. 

This implies a fortiori that there is no reason why $U_t + H_1(x,t,U(x,t),DU(x,t))$ should be
nonpositive and, since we can argue exactly in the same way with control $\alpha_2$ associated to
the control problem in $\Omega_2\times (0,\Tf)$, there is also no reason why $U_t + H_2(x,t,U(x,t),DU(x,t))$
should be nonpositive either. Hence, the Ishii subsolution inequality, namely
\begin{equation}\label{ineq.ishii.sub}
\min(U_t+H_1(x,t,U,DU),U_t+H_2(x,t,U,DU))\leq0\text{ on }\H\times(0,\Tf) \end{equation}
is not natural at all from the control point of view.

\

\noindent\textsc{Why the Ishii inequality actually holds ---}
The proof of Proposition~\ref{prop:ishii-Up} gives a first way to answer this puzzle in the case of
$\VFp$ (the argument would be exactly the same in the case of $\VFm$). 

On one hand, the $\HT$-inequality is natural since it shows that all the admissible trajectories which stay on
$\H$ are suboptimal. On the other hand, if $\VFp_t+H_1\leq0$, the Ishii inequality holds while if $\VFp_t +H_1>0$, 
the inequality $\VFp_t +\HT\leq 0$ implies that necessarily $\VFp_t +H_2\leq 0$, since the dynamics such that the
$X$-trajectories stay on $\H$ are convex combinations of the $b_1$ and $b_2$-ones. In any case we
obtain \eqref{ineq.ishii.sub} for $\VFp$.

Hence, the Ishii subsolution inequality holds on $\H\times(0,\Tf)$ as a consequence of the natural $\HT$-inequality. 
And one may wonder whether it is not more natural to define subsolution by just imposing the $\HT$-inequality on
$\H\times(0,\Tf)$, dropping \eqref{ineq.ishii.sub}.

This is exactly what the notion of Flux-Limited Solutions is doing, \cf Chapter~\ref{chap:FLSP}.
Indeed, as a by-product of the argument which leads to \eqref{eqn:ineq-sub-a1}, we have natural
$H_1^+$ and $H_2^-$ inequalities, at least for the value functions $\VFm,\VFp$. 

Moreover it is clear that this last remark remains valid in far more general cases: we have natural
subsolution inequalities for the controls for which the dynamics ``move away from the discontinuities''.

\

\noindent\textsc{General subsolutions, General discontinuities: the stratified case ---}
Maybe looking only at value functions is misleading since we know that Theorem~\ref{thm:SubP} holds
and maybe also that the two-domains case is a very particular situation regarding the Ishii
subsolution inequality on the discontinuity.

This suggests a more general question: for unnatural reasons, the $\F_*\leq 0$ inequality holds on
discontinuities for value functions; does this ``little miracle'' hold both for general subsolutions
and for more complicated discontinuities? 

Surprisingly the answer is yes in the stratified framework under suitable assumptions: in
Section~\ref{rweqs}, it is a consequence of a \LCR in the case of regular subsolutions and this
points out that such inequality always holds {\em for any regular subsolution} provided a comparison
result holds. Hence the $\F_*\leq 0$-inequality on discontinuities appears more as a
consequence than as a required inequality in the definition.

This is confirmed by the fundamental Lemma~\ref{lem:comp.fundamental} which is the keystone to prove
comparison results: this lemma is based on $(i)$ a ``tangential inequality'' on the discontinuity
(for example the $\HT$-inequality on $\H\times (0,\Tf)$) and $(ii)$ a subdynamic programming
principle for the subsolution {\em outside} the discontinuity (in $\Omega_1\times (0,\Tf)$ and in
$\Omega_2\times (0,\Tf)$ here).  None of these ingredients uses the Ishii subsolution inequality on
$\H\times (0,\Tf)$.

As a conclusion of this section, we can remark that, thanks to the above arguments, imposing or not
the Ishii subsolution inequality on the discontinuities is not a real issue: one way or the other,
it will hold at least in frameworks where a suitable comparison result holds.

But as the reader can notice everywhere in this book, even if it is not the only way to obtain it,
the Ishii subsolution inequality on $\H\times (0,\Tf)$ provides the regularity of subsolutions, a
fundamental ingredient. This is why we make the choice to maintain it most of the time.

\section{Infinite horizon control problems and stationary equations}

The aim of this section is to briefly describe the analogous results in the infinite horizon case where
the HJ Equation is stationary: we will only skim over this problem since all the results are not only
straightforward translations and adaptations of the finite horizon/evolution equations case but the proofs are
even simpler from a technical point-of-view. We recall that this case was studied in details in the works of
Briani and the authors of this book (cf. \cite{BBC1,BBC2}) and actually almost all the ideas and results of this
part appear for the first time in these two articles.

From the control point-of-view, we are given for $x\in \Omega_i$ and for $i=1,2$
$$ \BCL_i (x):=\{ (b_i(x, \alpha),c_i(x, \alpha),l_i (x, \alpha))\ : \ \alpha \in A\}\; ,$$
where, as above, the control set $A$ is a compact metric space and the $(b_i,c_i,l_i)$ are defined on $\R^N\times A_i$ and
satisfy \HCP. We assume, in addition, that there exists $\lambda>0$ such that, for $i=1,2$,
$$ c_i(x, \alpha)\geq \lambda \quad \hbox{in  }\R^N\times A\; .$$
As in the finite horizon case, we define $\BCL(x)$ as $\BCL_i(x)$ if $x\in \Omega_i$ and as the closed convex envelope of
$\BCL_1 (x)\cup\BCL_2(x)$ if $x\in \H$.

Using this $\BCL$, we can solve the differential inclusion equation for $(X,D,L)$
$$ (\dot X(s) ,\dot D(s) ,\dot L(s) )\in \BCL(X(s))\; ,$$
with $(X(0),D(0),L(0))=(x,0,0)$. We can also define ``regular'' and ``singular'' dynamics on $\H$, $\mT(x)$, $\mT^{reg}(x)$ and
the tangential Hamiltonians $\HT,\HT^{reg}$. The associated value functions are
$$\VFm(x):=\inf_{\mT(x)}\left\{\int_0^{+\infty} l(X(s),a(s))\exp(-D(s))\ds\right\}\, ,$$
$$\VFp(x):=\inf_{\mT^\reg(x)}\left\{\int_0^{+\infty} l(X(s),a(s))\exp(-D(s))\ds\right\}\, ,$$
where $l(X(s),a(s))$ is defined as in Theorem~\ref{thm:existence.traj}.

From the pde point-of-view, the related problem is
\begin{equation}\label{eq:HJ-stat}
\begin{cases}
H_1(x,u,Du)=0 & \hbox{in }\Omega_1\; ,\\
H_2(x,u,Du)=0 & \hbox{in }\Omega_2\; ,
\end{cases}
\end{equation}
with the standard Ishii inequalities on $\H$ where, for $i=1,2$,
$$
        H_i(x,r,p):= \sup_{\alpha \in A}\,\left\{-b_i(x, \alpha)\cdot p + 
        c_i(x, \alpha)r-l_i(x,\alpha)\right\}\; .
$$ 

The result is the following
\begin{theorem}\label{thm:Ishii-codim1-stat} Under the above assumptions,\\[-5mm]
    \begin{enumerate}
        \item[$(i)$] the value functions $\VFm,\VFp$ are well-defined and bounded. They are viscosity solutions of \eqref{eq:HJ-stat}.
        \item[$(ii)$] The value function $\VFm$ satisfies
\begin{equation}\label{eq:HT-stat}
\HT (x,u,Du)\leq 0 \quad \hbox{on  }\H\; ,
\end{equation}
while the value function $\VFp$ satisfies
\begin{equation}\label{eq:HTreg-stat}
\HT^{reg} (x,u,Du)\leq 0 \quad \hbox{on  }\H\; .
\end{equation}
\item[$(iii)$] The value function $\VFm$ is the minimal viscosity supersolution (and solution) of \eqref{eq:HJ-stat}, while $\VFp$
is the maximal viscosity subsolution (and solution) of \eqref{eq:HJ-stat}.
\item[$(iv)$] The value function $\VFm$ is the unique viscosity solution of \eqref{eq:HJ-stat} which satisfies \eqref{eq:HT-stat}.
    \end{enumerate}
\end{theorem}

We leave the proof of this theorem to the reader since, as we already wrote it above, it is a routine adaptation of the ideas
described in this part.

\section{Towards more general discontinuities: a bunch of open problems.}
\label{TMGD}

A very basic and minimal summary of Part~\ref{part:codim1}---including the previous section---can be expressed as follows: for 
Problem~\eqref{pb:half-space}, we are able to provide an explicit control formula for the minimal supersolution (and solution)
$\VFm$, and also an explicit control formula for the maximal (and solution) $\VFp$.

The next natural questions are: is it possible to extend such results to more general type of discontinuities? It can also be thought
that some of them are very particular cases which only appear because of the codimension $1$ discontinuity and that simpler results
may exist for higher codimensions because of some kind of ``eliminability property'' (?). This idea can only be reinforced by the fact that,
as we will see it in Part~\ref{part:NA}, $\VFp$ is the limit of the vanishing viscosity method.

Before coming back to this question of $\VFp$ or more precisely to the identification of the maximal subsolution, we consider
the case of $\VFm$, which may be perhaps considered as being the more natural solution from the control point of view. Here
the answer to the above question is yes and this is not so surprising since, by Corollary~\ref{VFm-minsup}, we know in a very general
framework that $\VFm$ is the minimal viscosity supersolution of the Bellman Equations,  therefore we already have a lot of informations
on $\VFm$.

In the Part~\ref{stratRN}, we provide a rather complete study of stratified solutions in $\R^N$ and then, in Part~\ref{S-BC} in general
domains, which are the natural generalization of $\VFm$ in the case when the codimension-$1$ discontinuity is replaced by discontinuities
on Whitney stratifications\index{Whitney stratification}. As in Section~\ref{sect:codimIa}, we characterize the stratified solution $\VFm$ as
the unique solution of a suitable problem with suitable viscosity inequalities. The methods which are used to study Ishii solutions, relying partly
on control arguments and partly on pde ones, can be extended to this more general setting and we will emphasize the (even more important) roles of the subsolution inequalities, normal controllability, tangential continuity...etc.

But the case of the maximal subsolution (and solution) $\VFp$ is more tricky and several questions can be asked, in particular
\begin{enumerate}
\item[$(i)$] Can one provide an explicit control formula for $\VFp$?
\item[$(ii)$] Is it still true that the vanishing viscosity method converges to $\VFp$? \index{Vanishing viscosity method!general questions on}
\end{enumerate}
Before describing the difficulties which appear even for rather simple configurations, we give a simple example which shows
that  we can definitively forget any hope on ``eliminability property''

\subsection{Non-uniqueness in the case of codimension $N$ discontinuities}

We consider the stationary equation
\begin{equation}\label{eq:exam-stat}
\vert Du - \frac{x}{|x|}\vert+u=|x|\quad\hbox{in  }\R^N\; ,
\end{equation}
for which we have only a discontinuity at $x=0$. The Ishii inequalities at $0$ read
$$ \min_{|e|=1}|Du-e| + u(0)\leq 0\; ,$$
$$ \max_{|e|=1}|Du-e| + u(0)\geq 0\; .$$

A first clear solution is $u_1(x)=|x|$ which is a smooth solution outside $0$ and, at $0$, the superdifferential of $u_1$ is empty
while the subdifferential is $\overline{B(0,1)}$ and the supersolution inequality obviously holds.

Now we look for an other solution of the form $u_2(x)=\varphi(|x|)$ for a smooth function $\varphi:[0,+\infty)\to \R$. Outside $0$,
$u_2$ is smooth and leads to the equation
$$ |\varphi'(s)-1|+\varphi(s)=s\; .$$
And $\psi(s)=\varphi(s)-s$ satisfies $|\psi'(s)|+\psi(s)=0$. If we assume that $\psi(0)=\lambda$ is given, we have by uniqueness for this $1-d$
HJ-Equation (assuming that $\psi$ is bounded), $\psi(s)=\lambda e^{-s}$ and this implies that $\lambda \leq 0$. This means that we have
a family $(u_2^\lambda)_{\lambda \leq 0}$ of candidates for being solutions of \eqref{eq:exam-stat}, where the $u_2^\lambda$ are given by
$$ u_2^\lambda(x) = |x|+\lambda e^{-|x|}\; .$$
First it is clear that $u_2^\lambda$ is a smooth solution outside $0$. At $0$, since $\lambda \leq 0$, the superdifferential of $u_2^\lambda$ is empty
while its subdifferential is $\overline{B(0,1-\lambda)}$. In particular $p=0$ is in the subdifferential of $u_2^\lambda$ and
$$ \max_{|e|=1}|0-e| + \lambda \geq 0\; .$$
This means that $\lambda \geq -1$ and all $\lambda \in [-1,0]$ gives a solution.

Hence we do not have uniqueness despite of this very high codimension of the singularity. Examining a little bit more carefully the above 
argument, it is easy to show that $u_1$ is the maximal subsolution (and solution) while $u_2^{-1}$ is the minimal supersolution (and solution)
of \eqref{eq:exam-stat} in the space of functions with sublinear growth: indeed, it suffices as above to consider that a solution of \eqref{eq:exam-stat} is a solution of the Dirichlet problem
$$
\vert Du - \frac{x}{|x|}\vert+u=|x|\quad\hbox{in  }\R^N\setminus{\{0\}}\; ,\; u(0)=\lambda\; ,$$
for which we have a comparison result. Then we notice that, by the equation, $\lambda \leq 0$ and the solution of this Dirirchlet problem
is necessarily given by $u_2^{\lambda}$ for some $\lambda \in [-1,0]$. 

Last but not least, we look at the associated control problem. Outside $0$, we have
$$b(x,\alpha)=\alpha \in \overline{B(0,1)}\; ,\; c(x,\alpha)=1\; ,\; l(x,\alpha)=|x|-\alpha \cdot \frac{x}{|x|}\; ,$$
and $\BCL(0)$ is obtained by computing the convex enveloppe. It is worth pointing out that the cost $|x|$ in
$l(x,\alpha)$ suggests that the best strategy consists in going to $0$ but a direct path from $x$ to $0$ would use
the control $\alpha=-\dfrac{x}{|x|}$ with a cost $|x|+1$ in $l(x,\alpha)$ because of $-\alpha \cdot \dfrac{x}{|x|}$-term.

This large cost of controls pointing toward $0$ is translated in terms of ``regular'' and ``singular'' strategies to stay at $0$:
a ``regular'' strategy can be thought as a convex combination of controls pointing toward $0$, i.e.
with $-\alpha \cdot \frac{x}{|x|}\geq 0$. Therefore the minimal cost for a ``regular'' strategy is $0$. But if we accept all convex
combination, we may use controls with $-\alpha \cdot \frac{x}{|x|}< 0$ and even $-\alpha \cdot \frac{x}{|x|}=-1$ coming from
two opposite directions $x$ and $-x$ at $0$. 

This explains the extremal value $\lambda=0$ and $\lambda =-1$ and $u_1$ is nothing but a $\VFp$ while $u_2^{-1}$ is nothing but
$\VFm$.

Last remark: in this case, the convergence of the vanishing viscosity method is easy to establish since $u_1$ is convex and therefore a 
subsolution for the vanishing viscosity equation. Hence the two half-relaxed limits for the vanishing viscosity approximation are
larger that $u_1$ but they are also between the maximal subsolution and the minimal supersolution of \eqref{eq:exam-stat}, i.e. 
$u_1$ and $u_2^{-1}$. Therefore they are both equal to $u_1$.

\subsection{Puzzling examples}\label{sec:puzzling}

In general, we are unable to give a control formula for the maximal subsolution of an HJB-Equation with discontinuities of
codimensions $>1$, and even in very simple examples. The problem is both to determine what is a ``regular'' strategy but
also to concretely prove that the associated value function is indeed the maximal subsolution.

In order to be more specific and to fix ideas,
we consider two interesting examples: the first one is the case when we still have two domains but the interface is not smooth, typically Figure~\ref{fig:openpb} below.
\begin{figure}[htp]
    \begin{center}
   \includegraphics[width=0.6\textwidth]{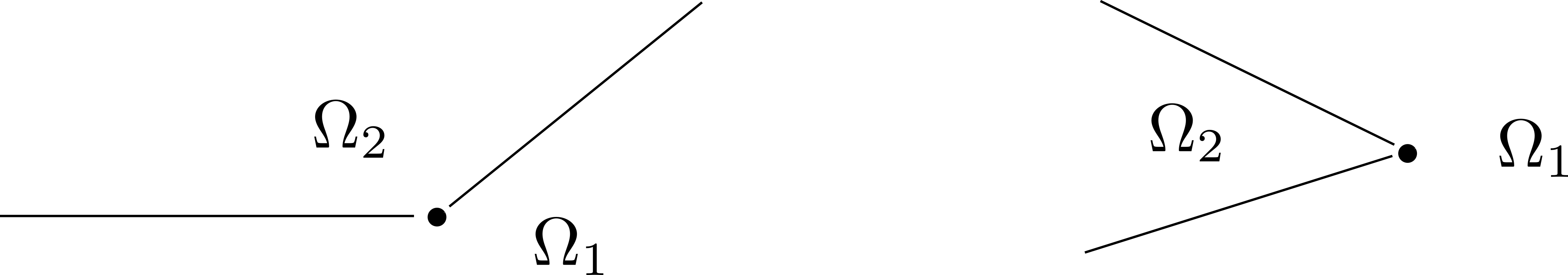}
   \caption{Two domains with a non-smooth interface}
   \label{fig:openpb} 
   \end{center}
\end{figure}

A second very puzzling example is the ``cross-case'' where $\R^2$ is decomposed into its four main quadrants, see Figure~\ref{fig:cross} below. And of course, one may also have in mind ``triple-junction configurations'' in between these two cases.
\index{Stratification!cross case}\label{pb:cross}

\begin{figure}[htp]
    \begin{center}
   \includegraphics[width=0.4\textwidth]{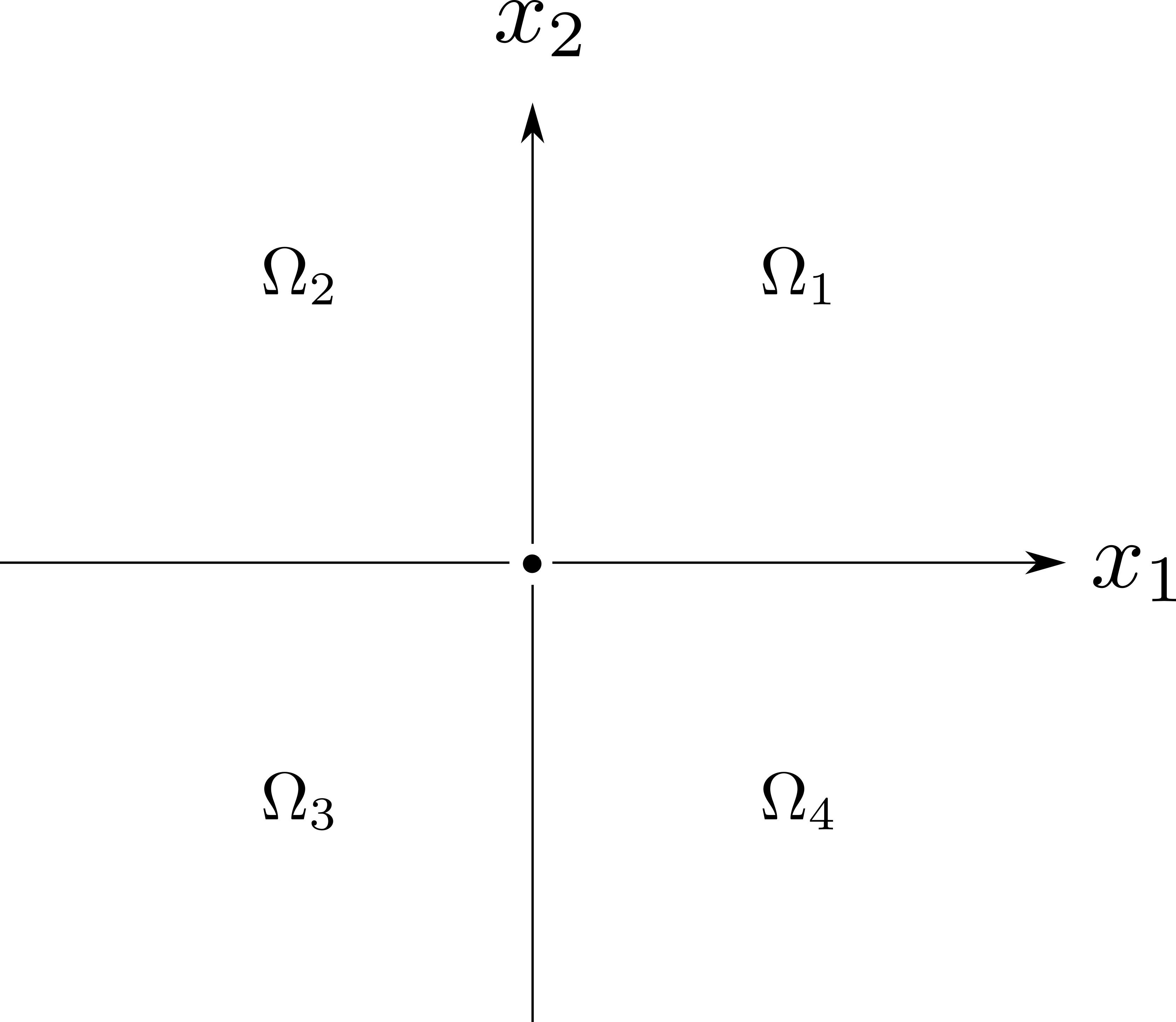}
   \caption{The cross-case}
   \label{fig:cross} 
   \end{center}
\end{figure}

\index{Vanishing viscosity method!general questions on}
The importance of the above questions is due to the numerous applications and we can mention for example front propagations phenomenas or Large Deviations type results: in both case, one has to identify the limit of the vanishing viscosity method and an ``action functional'' which exactly means to answer the above questions if the diffusions and/or drift involved in these problems are discontinuous. 

We refer for example to Souganidis \cite{S-FP} and references therein for the viscosity solutions' approach of front propagations
in reactions diffusion equations (like KPP (Kolmogorov-Petrovskii-Piskunov) type equations) and to Bouin \cite{Bouin} and references therein for front 
propagation in kinetic equations. For the viscosity solutions' approach of Large Deviations problems, we refer to \cite{BP3} (see also \cite{Ba}).

Now we turn to the questions (i) and (ii) of the beginning of Section~\ref{TMGD} which are largely open even in the two simple
cases described above. We first remark that most  of the results of this part, in particular those obtained by pde methods,
use in a crucial way the codimension-$1$ feature of the problem, via the  normal direction which determines which are the
inward and outward dynamics to the $\Omega_i$'s but also the $H_i^\pm$, and therefore the key 
$\HTreg$ Hamiltonian.

Concerning Question~$(i)$, in terms of control, the additional difficulty is to identify the ``regular strategies'' which allow to stay at the new discontinuity point 
($0$ in the cross-case) and then to show that using only these ``regular
strategies'', $\VFp$ is an Ishii solution of the problem. For Question~$(ii)$,
the proofs which are given above use either $\VFp$ (and therefore require an
answer to Question~$(i)$) or the codimension-$1$ feature of the problem via the Kirchhoff condition.

For all these reasons, even in the very simple configurations we propose above, we DO NOT know the right answer... but we hope that some readers will be able to find it!

In order to show the difficulty, we provide a ``simple'' result in the cross-case in $\R^2$, which DOES NOT give the result we wish but which uses the natural ingredients which should be useful to get it.

We are going to consider the problem
$$ u_t + H_i(Du)= 0\quad \hbox{in  }\Omega_i \times (0,\Tf)\;,\; \hbox{for $i=1,2,3,4$},$$
where the Hamiltonian $H_i$ are given by
$$H_i(p)=\sup_{\alpha_i\in A_i}\{-b_i(\alpha_i)\cdot p-l_i(\alpha_i)\}\;.$$
where $A_i$ are compact metric spaces. We are in a very simplified framework since we do not intend to provide
general results, so we also assume that the Hamiltonians $H_i$ are coercive, and even that there exists $\delta >0$ such that
$$B(0,\delta) \subset \{b_i(\alpha_i); \ \ \alpha_i \in A_i \}\quad \hbox{for any $i=1,2,3,4$}\; .$$
This is natural as a normal controllability assumption.

Of course, these equations in each $\Omega_i$ have to be complemented by the Ishii conditions on the two axes:
except for $x=0$, we are in the framework described in this part since we face a codimension $1$ discontinuity. Therefore we
concentrate on the case $x=0$ where, in order to identify $\VFp$, we have to identify the ``$\HTreg$'', i.e. the ``regular strategies''
which allow to remain at $x=0$.

In order to do so, we introduce the set $\A$ of controls $(\alpha_1, \alpha_2, \alpha_3,\alpha_4)$ such that, on one hand, $b_i(\alpha_i) \in D_i$ for $i=1,2,3,4$ where
$$ D_i=\{b_i(\alpha_i); \ b_i(\alpha_i)\cdot x \leq 0\;\hbox{for all  }x\in \Omega_i\}\; ,$$
and, on the other hand, there exists a convex combination of the $b_i(\alpha_i)$ such that
$ \sum_{i=0}^4 \,\mu_i b_i(\alpha_i)=0$. Such a convex combination may not be unique and we denote by $\Delta$ the set of all such convex combinations.

Finally we set
$$ \HT^\mathrm{reg-cross}:= \sup_{\A}\Big\{\inf_\Delta\Big(-\sum_{i=0}^4 \,\mu_i l_i(\alpha_i)\Big)\Big\}\; .$$
Notice that here, since we consider a zero-dimensional set, the Hamiltonian $\HT^\mathrm{reg-cross}$ reduces to a real number.
We have the
\begin{lemma}\label{cross} If $u: \R^2 \times (0,\Tf)\to \R$ is an Ishii subsolution of the above problem then
$$ u_t + \HT^\mathrm{reg-cross}\leq 0 \quad\hbox{on  }\{0\} \times (0,\Tf)\; .$$
\end{lemma}

\begin{proof}Let $\phi$ be a $C^1$ function on $(0,\Tf)$ and $\tb$ be a strict local maximum point of $u(0,t)-\phi(t)$. We have to show that
$\phi_t(\tb)+ \HT^{reg-cross}\leq 0$.

To do so, we consider $(\alpha_i)_i \in \A$ and, for $\delta>0$ small, we consider the affine functions
$$ \psi_i (p)= \phi_t(\tb) -b_i(\alpha_i)\cdot p-l_i(\alpha_i) -\delta \; .$$
Applying Farkas' Lemma, there are two possibilities; the first one is: there exists $\bar p$ such that $\psi_i (\bar p) \geq 0$ for all $i$. In that case,
we consider the function $(x,t) \mapsto u(x,t)-\psi(t) - \bar p \cdot x -\frac{|x|^2}{\e}$ for $0<\e \ll 1$. 

Since $\tb$ is a strict local maximum point of $u(0,t)-\phi(t)$, this function has a local maximum point at $(\xe,\te)$ and $(\xe,\te)\to (0,\tb)$ as $\e\to 0$. Wherever the point $\xe$ is, we have an inequality of the type
$$ \phi_t(\te) + H_i (\bar p + \frac{2x}{\e})\leq 0\; .$$
But if such $H_i$ inequality holds, this means that we are on $\Omegb_i$ and in particular
$$ \phi_t(\te)  -b_i(\alpha_i)\cdot (\bar p + \frac{2x}{\e})-l_i(\alpha_i)  \leq 0\; .$$
Recalling that $b_i(\alpha_i) \in D_i$, this implies
$$ \phi_t(\te)  -b_i(\alpha_i)\cdot \bar p - l_i(\alpha_i)  \leq 0\; .$$
For $\e$ small enough, this inequality is a contradiction with $\psi_i (\bar p) \geq 0$ and therefore this first case cannot hold.

Therefore, we are always in the second case: there exists a convex combination of the $\psi_i$, namely 
$\displaystyle \sum_{i=0}^4 \,\mu_i \psi_i$ which gives a negative number. In that case, it is clear that we have
$$ \sum_{i=0}^4 \,\mu_i b_i(\alpha_i)=0\quad\hbox{and} \quad \phi_t(\tb) -\sum_{i=0}^4 \,\mu_i l_i(\alpha_i) -\delta \leq 0\; .
$$ 
This implies that $$  \phi_t(\tb) + \inf_\Delta\Big(-\sum_{i=0}^4 \,\mu_i l_i(\alpha_i)\Big) -\delta \leq 0\; ,$$
and since this is true for any $(\alpha_i)_i \in \A$ and for any $\delta >0$, we have the result.
\end{proof}

The interest of this proof is to show the two kinds of arguments which seem useful to obtain an inequality for the subsolutions at $0$:
$(i)$ to find the suitable set $\Delta$ of ``regular strategies'' which allow to stay fixed at $0$; $(ii)$ to have suitable properties on the $b_i$'s 
which allow to  deal with the $2x/\e$-term  in the Hamiltonians, in other words we have to define suitable ``outgoing strategies''. 

Again this result is not satisfactory and we do not think that it leads to the desired result in the cross case.


\part{Hamilton-Jacobi Equations with Codimension One Discontinuities: the ``Network" Point-of-View}
\label{part:NA}
\fancyhead[CO]{HJ-Equations with Discontinuities: The Network Approach}


\chapter{Introduction}
\label{intro-NA}
\abstract{Despite looking at the same problem from a pde point-of-view, the approach of this part is
completely different and does not use any optimal control tool, just pure pde arguments. 
The first consequence of this different point-of-view is a change of test-functions. 
Two notions of solutions (flux-limited solutions \`a la Imbert-Monneau and junction
viscosity solutions \`a la Lions-Souganidis) are described in this part with all their stability and
comparison properties. They are associated to two different types of conditions at the interface.
This introduction describes them with the assumptions they should satisfy.}

Contrarily to Part~\ref{part:codim1} where the question of a codimension $1$ discontinuity in
Hamilton-Jacobi Equations is mainly addressed in the case of convex Hamiltonians by using control
arguments, the aim of this part is to describe several complementary pde points-of-view which allow
to obtain more general results, and most of them for non-convex equations.
However we often choose to present them in the framework of Part~\ref{part:codim1} for justifying
the assumptions we use and showing the interest of the results.

\section{The ``network approach'': a different point-of-view}

In order to present these other pde approaches, let us focus first on a simple $1$ dimensional
configuration, the terminology ``network point-of view'' originating from this situation.
Considering an Hamilton-Jacobi Equation with a discontinuity at $x=0$, we have in mind the picture
in Fig.~\ref{fig:ishii.na} below

\begin{figure}[htp]
    \begin{center}
    \setlength{\unitlength}{1.6cm}
    \begin{picture}(4,0.8)(0,0.7)
    \thinlines
    \put(0,1){\vector(1,0){4}}
    \put(1.7,0.95){\small $\bullet$}
    \put(3.8,0.8){\small $x$}
    \put(0.5,1.1){\small $H_2=0$}
    \put(2.5,1.1){\small $H_1=0$}
    \end{picture}
    \caption{The Ishii point of view}
    \label{fig:ishii.na}
    \end{center}
\end{figure}
Here, Ishii's definition of viscosity solutions in $\R$ is quite natural and involves
$\min(H_1,H_2)$ and $\max(H_1,H_2)$ at $x=0$.

But, since the equations are different in the sets $\{x>0\}$ and $\{x<0\}$, we can see as well the
picture as two segments joining at $x=0$:   

\

\begin{figure}[htp]
    \begin{center}
    \setlength{\unitlength}{1.6cm}
    \begin{picture}(4,1.5)(0,1)
    \thinlines
    \put(1.7,1){\vector(-1,0){2}}
    \put(1.7,0.95){\small $\bullet$}
    \put(1.73,0.98){\vector(1,1){1.5}}
    \put(-0.3,0.8){\tiny $x_2$}
    \put(3.2,2.2){\tiny $x_1$}
    \put(0.5,1.1){\small $H_2=0$}
    \put(2.5,1.4){\small $H_1=0$}
    \end{picture}
    \caption{The network point of view}
    \label{fig:na.na}
    \end{center}
\end{figure}
Now, $J_1=\{x>0\}$ and $J_2=\{x<0\}$ become two different branches of a (simple) network and it
becomes natural to introduce adapted coordinates on $J_1,J_2$, which are nothing but $x_1 = x$ on
$J_1$ and $x_2=-x$ on $J_2$.

\subsection{A larger space of test-functions}

The first main consequence of this different point of view is that the ``natural'' test-functions
are not the same as in the Ishii approach since they can be chosen differently in $J_1$ and $J_2$,
with just a continuity assumption at $x=0$. 

In our original framework in $\R^N$ with $\Omega_1, \Omega_2, \H$ introduced in
Section~\ref{sect:stab} where an analogous remark holds, just replacing $J_1$ by $\Omega_1$, $J_2$
by $\Omega_2$ and $0$ by $\H$, this suggests the space of ``natural'' test-functions as
\label{not:PCun}
\begin{definition}
    We denote by $\PC1$ the space of piecewise $C^1$-functions $\psi \in C(\R^N \times [0,\Tf])$
    such that there exist $\psi_1 \in C^1(\Omegb_1\times [0,\Tf])$, $\psi_2 \in C^1(\Omegb_2\times [0,\Tf])$ such that $\psi = \psi_1$
    in $\Omegb_1\times [0,\Tf]$ and $\psi =
    \psi_2$ in $\Omegb_2\times [0,\Tf]$.
\end{definition}

An important point in this definition is that $\psi=\psi_1=\psi_2$ on $\H\times [0,\Tf]$ and $D_\H
\psi = D_\H \psi_1= D_\H \psi_2$ on $\H\times [0,\Tf]$, $\psi_t =(\psi_1)_t= (\psi_2)_t $ on
$\H\times [0,\Tf]$.  We recall here that $D_\H$ is the tangential derivative.

This change of test-functions is a first step but it remains of course to examine the kind of
``junction condition'' we can impose on $\H\times [0,\Tf]$, since, contrarily to what happens for
the Ishii definition, no obvious choice seems to stand out. 

The first attempt could be to try the standard Ishii inequalities with this larger set of
test-functions with the convention (since the test-functions are not necessarily smooth
on $\H\times [0,\Tf]$) to use the derivatives
of $\psi_1$ in the $H_1$-inequalities and those of $\psi_2$ in the $H_2$-inequalities.
On the simplest example where the equations are 
\[ \tag{HJ-gen} 
\begin{cases}
u_t+H_1 (x,t,u,Du) = 0 \quad\hbox{in  }\Omega_1\times (0,\Tf)\; , \\
u_t+H_2 (x,t,u,Du) = 0 \quad\hbox{in  }\Omega_2\times (0,\Tf)\; ,
\end{cases} 
\]
and without additional Hamiltonian on $\H$, these conditions are
\[
\begin{cases}
\min(u_t+H_1 (x,t,u,Du),u_t+H_2 (x,t,u,Du))\leq 0 \quad\hbox{on  }\H \times (0,\Tf)\; , \\
\max(u_t+H_1 (x,t,u,Du),u_t+H_2 (x,t,u,Du))\geq 0 \quad\hbox{on  }\H \times (0,\Tf)\; .
\end{cases} 
\]
But it is easy to check that, with test-functions in $\PC1$, there is no subsolutions if $H_1,H_2$
are both coercive. The argument is the following: if $u-\varphi$ has a maximum at some point
$(0,t)\in\H\times(0,\Tf)$, then $u-(\varphi+C|x_N|)$ also has a maximum at the same point and since
$\varphi_C(x,t):=\varphi(x,t)+C|x_N|$ belongs to  $\PC1$ we can use it to test the inequalities. But,
since the Hamiltonians are coercive, taking $C>0$ large enough yields an impossibility
since both $|D(\varphi_C)_1(x,t)|$ and $|D(\varphi_C)_2(x,t)|$ can be taken as large as we wish.

\subsection{Different types of junction conditions}
\label{sect:types.junctions}

As a consequence of the simple remark above, it is clear that the question of the right junction
condition to be imposed on $\H$ becomes crucial. And it obviously depends on the type of
applications we have in mind.

\

\noindent\textbf{(a)} \emph{Flux-limited condition} ---
From Chapter~\ref{chap:Ishii}, it seems obvious that in the framework of control problems, a natural
contidition on $x=0$ is the following
\[ \tag{FL} 
u_t+G (x,t,u,D_\H u) = 0 \quad\hbox{on  }\H\times (0,\Tf)\; . \\
\]
Indeed, for applications to optimal control, one may have in mind a specific control on $\H$, i.e. a
specific dynamic, discount and cost as in Chapter~\ref{sec:H0.case}. In the network literature
(cf. Imbert and Monneau \cite{IM,IM-md,IN}), the associated terminology is ``flux-limited
condition'' (See Section~\ref{Traffic} for a partial justification of this terminology). Concrete
modellings and applications lead to a variety of different flux-limited conditions at the boundary,
including more general ones 
\[ \tag{GFL} 
G (x,t,u,u_t,D_\H u) = 0 \quad\hbox{on  }\H\times (0,\Tf)\; ,\\
\]
where $G$ satisfies: there exists $\gamma>0$ such that, for any $x\in \H$, $t\in [0,\Tf]$, $r\in \R$,
$p'\in \H$ and $a_2\geq a_1$, one has 
\begin{equation}\label{basic-hyp-GFL}
        G(x,t,r,a_2,p')-G(x,t,r_2, a_1,p') \geq \gamma(a_2-a_1) \; .
\end{equation}
In fact, if \eqref{basic-hyp-GFL} holds, it is a simple exercise to show that there exists $\tilde
G$ such that $G(x,t,r,a,p')$ and $a+\tilde G(x,t,r,p')$ have the same signs. In other words, a
general flux-limited condition \GFL is equivalent to a simple flux-limited condition \FL (both for
the sub and supersolution condition), and pushing the exercise a little bit further, the reader will
notice that the assumptions on $G$ can be transfered without any difficulty to $\tilde G$.

For this reason, in the sequel we focus on the study of Conditions~\FL but either by doing the
above exercise or repeating readily the arguments, it will be clear that all the definitions and
results extend without any difficulty to \GFL.

\

\noindent\textbf{(b)} \emph{Kirchhoff type conditions} ---
This second type of condition involves the normal derivatives of the solution on $\H$. The simplest
one, used in various applications and in particular for networks, is the {\em Kirchhoff
condition}\index{Kirchhoff condition}
\begin{equation}\label{NA-KC}\tag{KC} 
    \frac{\partial u}{\partial n_1}+\frac{\partial u}{\partial n_2}=0
    \quad \hbox{on  }\H\times (0,\Tf)\; ,
\end{equation}
where, for $i=1,2$, $n_i(x)$ denotes the unit normal to $\domeg_i$ pointing outward $\Omega_i$ at
$x\in \domeg_i$. 

\

\noindent\textbf{(c)} \emph{General junction conditions} ---
More generally, a junction type condition may have the form
\begin{equation}\label{GJC}\tag{GJC}
    G\Big(x,t,u,u_t,D_\H u, \frac{\partial u}{\partial n_1},\frac{\partial u}{\partial n_2}\Big)=0
    \quad \hbox{on  }\H\times (0,\Tf)\; ,
\end{equation}
where $G(x,t,r,a, p',b,c)$ has at least to satisfy the following monotonicty assumption: 
there exists $\alpha, \beta\geq 0$ such that, for any $x\in \H$, $t\in (0,\Tf)$, $r_1\geq r_2$,
$p'\in \H$, $a_1\geq a_2$, $b_1\geq b_2$, $c_1\geq c_2$, 
\begin{equation}\label{basic-hyp-GJC}
    \begin{aligned}
        G(x,t,r_1,a_1,p',b_1,c_1)\,-\, & G(x,t,r_2, a_2,p',b_2,c_2) \\[1mm] & \geq \alpha(a_1-a_2) + 
    \beta(b_1-b_2)+ \beta(c_1-c_2)\;.
    \end{aligned}
\end{equation}
In the sequel, we will often drop the dependence in $r$ in junction condition \GJC, just to
simplify a little bit the technicalities. But taking into account such dependence with a suitable
monotonicity assumption does not cause major problems. Precise assumptions are given in next
section.

\

Roughly speaking, each of these conditions is treated in the literature by using a different notion
of solution. In the case of \FL-conditions, and in particular if one has in mind applications to
control problems, the natural notion of solutions is the ``Flux-Limited solutions'', which
is introduced and extensively studied in \cite{IM,IM-md,IN}. However, this kind of solution is not
well-adapted for dealing with Kirchhoff type conditions, where a notion of ``Junction viscosity
solution'' is needed. 

This second notion of solution, rather similar to classical viscosity solutions is called ``relaxed
solutions'' in \cite{IM} and extensively used in the works of Lions and Souganidis \cite{LiSo1,
LiSo2}.

\section{The ``good assumptions'' used in Part~\ref{part:NA}}

In this part, most of the results we present are obtained using PDE methods. For this reason, the
control interpretation, and therefore the convexity of the Hamiltonians, is not playing a key role.
Depending on the chapter or the section, we are going to consider either convex, quasi-convex or
merely continuous Hamiltonians. This is why depending on the context we have to translate in this
section the ``good framework for HJ-Equations with discontinuities'' in the particular case of a codimension $1$
discontinuity already discussed in Section~\ref{sec:GA-codim1}.\index{Good framework for HJE!for a codimension $1$ discontinuity}

\

We refer first the reader to Section~\ref{sec:BasicA} where Basic Assumptions \HBACP and \HBAHJ
are defined. Then, in order to satisfy \Mong, we denote by \\[3mm]
\HBAHJp\ :\label{page:HBAHJp} assumption \HBAHJ in which we assume $\gamma(R) \geq 0$ for any $R$. \\[2mm]
\HBACPp\ : assumption \HBACP in which we assume $c(x,t,\alpha) \geq 0$ for any $x,t,\alpha$.

\smallskip

These reductions are only done in order to simplify matters, in any case a change $u\to u\exp(Kt)$
for a suitable constant $K$ allows to reduce to the above assumptions. We also point out that, thanks
to Chapter~\ref{sect:htc}, \HBAHJp and \HBACPp imply \LOCa, \LOCb because of the Lipschitz
continuity in $p$ of the Hamiltonians.

\subsection{Good assumptions on $H_1,H_2$}

We need here to translate the normal controllability and tangential continuity assumptions to the
case of general Hamiltonians:

\smallskip

\label{page:NCHJ}
\begin{assumption}{\NCHJ}{Normal controllability for general Hamiltonians.}
    For any $R>0$, there exists constants $C^R_2,C^R_3,C^R_4>0$ such that, for any $(x,t)\in
    \H\times (0,\Tf)$ with $|x|\leq R$, $|u|\leq R$ and $p=(p',p_N)$ with $p'\in \R^{N-1}$, $p_N\in
    \R$, $$H(x,t,u,p)\geq C^R_2 |p_N| - C^R_3|p'| -C^R_4\;.$$
\end{assumption}

\vspace*{-1cm}

\label{page:TCsH} 
\begin{assumption}{\TCsH}{Tangential Continuity for general Hamiltonians.} 
    For any $R>0$, there exists $C^R_1>0$ and a modulus of continuity $m^R : [0,+\infty[ \to
    [0,+\infty[$ such that for any $x=(x',x_N), y=(y',x_N)$ with $|x|, |y|\leq R$, $|x_N|\leq
    R^{-1}$, $t,s\in [0,\Tf]$, $|u|\leq R$, $p=(p',p_N)\in \R^{N}$,
    $$|H(x,t,u,p)- H(y,s,u,p)|\leq C^R_1(|x'-y'|+|t-s|)|p'| +m^R\big(|x'-y'|+|t-s|\big)\;.$$
\end{assumption}

With these assumptions we can formulate several ``good assumptions'' depending on the context:\\[3mm]
\GAGen \label{page:GAGen}\emph{General case ---} $H_1,H_2$ satisfy \HBAHJp and \NCHJ.\\[2mm]
\GAConv \label{page:GAConv}\emph{Convex case ---} $H_1,H_2$ satisfy \GAGen and are convex in $p$.\\[2mm]
\GAQC \label{page:GAQC}\emph{Quasi-convex case ---} $H_1,H_2$ satisfy \GAGen and \HQC. \\[2mm]
\GACC \label{page:GACC}\emph{Control case ---} \HBACPp and \NCoH are satisfied.

\begin{remark} 
    A priori, the variable $t$ being a ``tangential variable'', \TCsH should be formulated with a
    right hand side like $C^R_1(|x'-y'|+|t-s|)(|p'|+|p_t|)$ instead of $C^R_1(|x'-y'|+|t-s|)|p'|$;
    but since $H$ does not depend on $p_t$, the above formulation seems more natural. However, using
    the equation which gives $p_t=-H$, it is probably possible to change this assumption into the
    more general one, including a term like $C^R_1(|x'-y'|+|t-s|)(|p'|+\max(|H(x,t,u,p)|, |H(y,s,u,p)|))$
    in the right-hand side of \TCsH. We leave this open question to the reader.
\end{remark}

\subsection{Good assumptions on the junction condition}

We now turn to the assumptions on the function $G$ which appears in \FL or \GJC, recalling that we
are assuming it is independent of $r$ for simplicity. To do so, we first formulate a continuity
requirement, where the role of $\e_0$ will be clear later on. 

\medskip

\label{page:HContG}
\begin{assumption}{\HContG}{Continuity on the interface.}
    For any $R>0$, there exist constants $C^R_5, C^R_6$ such that, for any $x,y \in \H$, $t,s\in
    [0,\Tf]$, $|r|\leq R$, $p'_1,p'_2\in \R^{N-1}$, $a,b,c, a',b',c' \in \R$ 
    $$ |G(x,t,a, p_1',b,c)-G(y,s,a,p'_1,b,c)| \leq C^R_5 (|x-y|+|t-s|)\big(1+ |p_1'| +
    \e_0(|a|+|b|+|c|)\big)\; .$$ $$ |G(x,t,a', p_2',b',c')-G(x,t,a,p_1',b,c)| \leq C^R_6
    (|p_2'-p'_1| + (|a'-a|+|b'-b|+|c'-c|)\big)\; .$$
\end{assumption}

\vspace*{-0.5em}

The ``Good Assumptions'' on $G$ in the various cases are then the following

\label{page:GAGFL}
\begin{assumption}{\GAGFL}{Flux~limiter.}
$G$ is independent of $a,b,c$ and \HContG\  holds with
$\e_0=0$.
\end{assumption}

\label{page:GAGKT}
\begin{assumption}{\GAGKT}{Kirchhoff type.}
 \HContG  holds with $\e_0=0$ and \eqref{basic-hyp-GJC}  holds
with $\alpha \geq 0$, $\beta>0$.
\end{assumption}

\label{page:GAGFLT}
\begin{assumption}{\GAGFLT}{Flux-limited type.}
$G(x,t,a, p',b,c)=G_1(a, p',b,c) +G_2(x,t,a, p')$ 
where $G_1$ is a Lipschitz continuous function which satisfies \eqref{basic-hyp-GJC} with $\alpha > 0$, $\beta=0$ while 
$G_2$ satisfies \GAGFL.
\end{assumption}

The first two assumptions seem relatively natural, only the third one requires some comments: in
order to provide comparison results for the general junction condition \GJC, we are going to present
the Lions-Souganidis approach which is based on a ``tangential regularization'' of both the sub and
supersolution in the spirit of Sections~\ref{sec:regplus-subsol} and~\ref{sec:regul-supersol}. While
we are able to perform these regularizations in a rather general setting if $G$ is of ``Kirchhoff
type'' since \eqref{basic-hyp-GJC}\  holds with $\beta>0$, this is not the case anymore if
\eqref{basic-hyp-GJC}\  holds only with $\beta=0$. For this reason, we need \GAGFLT\  which is
(roughly speaking) the analogue of \TCs.

\section{What do we do in this part?} 

In the next two chapters of this part, we successively describe the notions of ``Flux-Limited Solutions'' and ``Junction
Viscosity Solutions'', and their properties. For each of them, we provide 
\begin{enumerate}
    \item [$(i)$] a general comparison result;
    \item[$(ii)$] a stability result;
    \item[$(iii)$] a convergence result of the vanishing viscosity method by specific arguments
        related to the corresponding notion of solution. 
\end{enumerate}

Moreover, for ``Flux-Limited solutions'', we also describe the connections with control problems.

It is worth pointing out that the notion of ``Junction Viscosity Solutions'' and the arguments of
Lions and Souganidis \cite{LiSo1, LiSo2} allow to obtain results which are valid without any
convexity assumption on the Hamiltonians, and in particular a very general comparison result,
despite some limitations due to \TC.  The theory for this notion of solutions is quite complete,
with very natural stability properties because of a definition which is very similar to the standard
viscosity solutions one.

Despite being very different, we prove in Chapter~\ref{sect:equiv.sols} that these notions of
solutions are ``almost equivalent'' in the case of flux-limited conditions \FL, at least in the
framework of quasi-convex Hamiltonians. We wrote ``almost'' because flux-limited subsolutions are
automatically regular as an easy consequence of their definition, while this is not the case for
junction viscosity subsolutions in general. Hence, complete equivalence holds if we assume that
the junction viscosity subsolutions are regular---which is true for instance in the case of
Kirchhoff conditions---.

In Chapter~\ref{sect:equiv.sols}, we provide the characterizations of the maximal and minimal
Ishii solutions in terms of other solutions. Last but not least, we show that junction viscosity sub
and supersolutions of various general junction conditions \GJC of Kirchhoff type are flux-limited
sub and supersolutions. The associated ``flux~limiter'' can be identified explicitly in terms of the
Hamiltonians $H_1,H_2$ of the equations in $\Omega_1,\Omega_2$ and of the nonlinearity of the
general junction conditions. These connections between general junction conditions \GJC of
Kirchhoff type and flux-limited conditions were extensively studied in \cite{IM,IM-md,IN} and they
are quite important because they allow to take advantage of the good stability properties of
``Junction Viscosity Solutions'' and the good connections of ``Flux-Limited Solutions'' with control
problems at the same time. The applications to the vanishing viscosity method and to the KPP problem
shows the efficiency of this machinery.

We conclude this part by a chapter describing all the results in a simple $1$-d framework very
similar to the scalar conservation law and then by various remarks on possible extensions or open
problems.

\chapter{Flux-Limited Solutions for Control Problems and Quasi-Convex Hamiltonians}
\label{chap:FLSP}

\abstract{This chapter is devoted to study flux-limited solutions \`a la Imbert-Monneau for
quasi-convex Hamiltonians: definition, stability and comparison properties are described in details}

In the control case, as it is clear from Chapter~\ref{chap:Ishii}, one may have in mind a specific
control problem on $\H$, i.e. a specific dynamic, discount and cost as in
Section~\ref{sect:codimIa}. In this setting, the most natural condition on $\H\times(0,\Tf)$ 
takes the form
\[ \tag{FL} 
u_t+G (x,t,u,D_\H u) = 0 \quad\hbox{on  }\H\times (0,\Tf)\; , \\
\]
which is called a ``flux-limited condition'' in the network literature (cf. Imbert and Monneau
\cite{IM,IM-md,IN}). Concrete modellings and applications lead to a variety of different
flux-limited conditions at the boundary, expressed as specific functions $G$.

\section{Definition and first properties}
\label{sec:def-FLS}

Let us first turn to the definition of ``flux-limited sub and supersolutions'' which requires the
introduction of some notations. 

\noindent \emph{In the case of control problems}, for $i=1,2$ the
Hamiltonians are given by
\be  \label{def:Ham-i}
    H_i(x,t,r,p):=\sup_{ \alpha_i \in A_i} \apg  -b_i(x,t,\alpha_i) \cdot p 
    +c_i(x,t,\alpha_i)r - l_i(x,t,\alpha_i)  \chg\;.
\ee
We then set $A_i^-:=\{ \alpha_i \in A_i\:: \: b_i(x,t,\alpha_i) \cdot e_N  \leq  0\}$ and
similarly $A_i^+:=\{ \alpha_i \in A_i\:: \: b_i(x,t,\alpha_i) \cdot e_N  >  0\}$, then
\be  \label{def:Ham-i-}
    H_i^-(x,t,r,p):=\sup_{\alpha_i\in A_i^-} 
    \apg  -b_i(x,t,\alpha_i) \cdot p +c_i(x,t,\alpha_i)r - l_i(x,t,\alpha_i)  \chg\;,
\ee
\be  \label{def:Ham-i+}
    H_i^+(x,t,r,p):=\sup_{\alpha_i\in A_i^+} 
    \apg  -b_i(x,t,\alpha_i) \cdot p +c_i(x,t,\alpha_i)r - l_i(x,t,\alpha_i) \chg\;.
\ee

Notice that the $+/-$ notation refers to the sign of $b_i \cdot e_N$ in the supremum, which implies
that $H^-_i$ ($i=1..2$) is nondecreasing with respect to $p_N$ (the normal gradient variable) while
the $H^+_i$ is nonincreasing with respect to $p_N$.

Finally, for the specific control problem on $\H$, we define for any $x\in\H$, $t\in [0,\Tf]$, $r\in
\R$, and $p_\H\in\R^{N-1}$
\be\label{def:G}
    G(x,t,r, p_\H):=\sup_{\alpha_0\in A_0}\{-b_0(x,t,\alpha_0)\cdot p_\H+
    c_0(x,t,\alpha_0)r-l_0(x,t,\alpha_0)\}\;.
\ee
For $i=1...2$, $b_i,c_i,l_i$ are at least bounded continuous functions defined on $\Omegb_i\times
[0,\Tf] \times A_i$ and $b_0,c_0,l_0$ are also bounded continuous functions defined on  $\H \times
[0,\Tf] \times A_0$. Therefore $H_1,H_2$ and $G$ are continuous.

\noindent \emph{In the case where the Hamiltonians are quasi-convex in $p$},
Section~\ref{sec:QCH-def} provides us with a definition of $H_i^+(x,t,r,p)$, $H_i^-(x,t,r,p)$ and we
assume that these functions and $G$ are continuous.

With these notations, we can give the definition of flux-limited viscosity sub and supersolutions
---\FLSub and \FLSup in short:
\index{Flux-limited solutions!definition} 

\begin{definition}\emph{--- Flux-limited solutions for quasi-convex
    Hamiltonians.}\label{defiFL}
    \begin{enumerate}
        \item[$(i)$] A locally bounded function $u: \R^N \times (0,\Tf) \rightarrow \R$ is a
            \FLSub of \HJgen-\FL if it is a classical viscosity subsolution of \HJgen and if, for
            any test-function $\psi \in \PC1$ and any local maximum point $(x,t) \in \H\times
            (0,\Tf) $ of $u^*-\psi$ in $\R^N\times (0,\Tf)$, at $(x,t)$ the following inequality
            holds 
            \[
              \max \Big(\psi_t + G(x,t,u^*,D_\H \psi), \psi_t+ H^+_1(x,t, u^*, D\psi_1) ,  
              \psi_t+ H^-_2(x,t, u^*, D\psi_2) \Big) \leq 0  \; ,
            \]
            where $u^*=u^*(x,t)$.

        \item[$(ii$)]  A locally bounded function $v: \R^N \times (0,\Tf) \rightarrow \R$
                is a \FLSup of \HJgen-\FL if it is a classical viscosity supersolution of \HJgen and
                if, for any test-function $\psi \in \PC1$ and any local minimum point $(x,t) \in
                \H\times (0,\Tf) $ of $v_*-\psi$ in $\R^N\times (0,\Tf)$, at $(x,t)$ the following
                inequality holds
            \[
              \max \Big(\psi_t + G(x,t,v_*,D_\H \psi), \psi_t+ H^+_1(x,t, v_*, D\psi_1) ,  
              \psi_t+ H^-_2(x,t, v_*, D\psi_2) \Big) \geq 0  \; ,
            \]
            where $v_*=v_*(x,t)$.
        \item[$(iii$)] A locally bounded function is a flux-limited solution if it is both a \FLSub and a \FLSup.
    \end{enumerate}
\end{definition}

Several remarks have to be made on this definition which is very different from the classical ones:
first we have a ``max'' both in the definition of supersolutions AND subsolutions; then we do not
use the full Hamiltonians  $H_i$ in the junction condition on $\H$ but $H^+_1$ and $H^-_2$. These
changes are justified when looking at the interpretation of the viscosity solutions inequalities in
the optimal control framework. Indeed
\begin{enumerate}
\item[$(i)$] the subsolution inequality means that any control is sub-optimal, $i.e.$ if one tries
to use a specific control, the result may not be optimal. But, of course, such a control has to be
associated with an ``admissible'' trajectory: for example, if we are on $\H$, a ``$b_1$'' pointing
towards $\Omega_2$ cannot be associated to a real trajectory, therefore it is not ``admissible'' and
this is why we use $H^+_1$. And an analogous remark justifies $H^-_2$. Finally the ``$\max$'' comes
just from the fact that we test all sub-optimal controls.  
\item[$(ii)$] Analogous remarks hold for the supersolution inequality, except that this inequality
    is related to the optimal trajectory, which has to be admissible anyway.
\end{enumerate}

With these remarks, the reader may be led to the conclusion that an ``universal'' definition of
solutions of \HJgen with the condition \FL can hardly exist: if we look at control problems where
the controller tries to maximize some profit, then the analogue of the $H^+_1$, $H^-_2$ above seem
still relevant because of their interpretation in terms of incoming dynamics but the $\max$ should
be replaced by $\min$ in both the definitions of sub and supersolutions.  Therefore it seems that
such particular definitions have to be used in each case since, again, the Kirchhoff condition does
not seem natural in the control framework.

As in the case of classical Ishii sub and supersolutions, we can define \FLSub and \FLSup using the
notions of sub and superdifferentials. We refer the reader to Section~\ref{sect:eqn-on-b} for the
introduction of these notions and various properties. Following this section, for $i=1,2$, we denote
by $Q_i=\Omega_1\times (0,\Tf)$ and $\overline{Q_i}^\ell=\Omegb_1\times (0,\Tf)$. 
As in Section~\ref{sect:eqn-on-b}, we restrict ourselves to the case of \usc
subsolution and \lsc supersolutions to simplify the notations but, in the general case, these
results have to be reformulated with either the \usc envelope of the subsolution or the \lsc
envelope of the supersolution.

\begin{proposition}\label{defiFL2}\emph{--- Flux-limited viscosity solutions via sub
    superdifferentials.}\smsp
        \noindent An \usc, locally bounded function $u: \R^N \times (0,\Tf) \rightarrow \R$ is a
    \FLSub of \HJgen-\FL if and only if
    \begin{enumerate}
        \item[$(i)$] for any $(x,t) \in Q_i$ ($i=1,2$) and for any $(p_x,p_t)\in D_{\overline{Q_i}^\ell}^+ u (x,t)$
    $$
    p_t + H_i (x,t,u(x,t),p_x) \leq 0,$$
    \item[$(ii)$] for any $(x,t) \in \H\times (0,\Tf)$ and for any $p_\H \in \H$, $p_1,p_2,p_t \in \R$
    such that $((p_\H,p_i),p_t)\in D_{\overline{Q_i}^\ell}^+ u (x,t)$ for $i=1,2$, noting $u=u(x,t)$,
    $$
    \max \Big(p_t + G(x,t,u,p_\H), p_t+ H^+_1(x,t, u, p_\H+p_1e_N) ,
    p_t+ H^-_2(x,t, u, p_\H+p_2 e_N) \Big) \leq 0 \;.
    $$
    \end{enumerate}
      
\noindent  A \lsc, locally bounded function $v: \R^N \times (0,\Tf) \rightarrow \R$
    is a \FLSup of \HJgen-\FL if and only if 
    \begin{enumerate}
        \item[$(i)$] for any $(x,t) \in Q_i$ ($i=1,2$) and for any $(p_x,p_t)\in D_{\overline{Q_i}^\ell}^-v (x,t)$
    $$
    p_t + H_i (x,t,v(x,t),p_x) \geq 0 \;,$$
    \item[$(ii)$] for any $(x,t) \in \H\times (0,\Tf)$ and for any $p_\H \in \H$, $p_1,p_2,p_t \in \R$
    such that $((p_\H,p_i),p_t)\in D_{\overline{Q_i}^\ell}^-v (x,t)$ for $i=1,2$, noting $v=v(x,t)$,
    $$
    \max \Big(p_t + G(x,t,v,p_\H), p_t+ H^+_1(x,t, v, p_\H+p_1e_N) ,  
    p_t+ H^-_2(x,t, v, p_\H+p_2 e_N) \Big) \geq 0 \;.
    $$
    \end{enumerate}
\end{proposition}

We omit the proof of Proposition~\ref{defiFL2} since it is an easy consequence of
Lemma~\ref{subdiff-hp-Omega} and Lemma~\ref{diff-twod}. As we already remark after the statement
of Lemma~\ref{diff-twod}, we point out that this equivalent definition via sub and
superdifferentials allows to show that, instead of using general $\mathrm{PC}^1$ test-functions, we
may consider only test-functions of the form $\chi(x_N)+\varphi (x,t)$ where $\chi \in {\rm PC}^1(\R)$
and $\varphi \in C^1(\R^N \times (0,\Tf))$. The reader will notice that we mainly use test-function
of this form in comparison proof, but this property is also useful to simplify the
proofs of several results.  

\begin{remark}
    Definition~\ref{defiFL} provides the notion of  ``flux-limited viscosity solutions'' for a
    problem with a codimension $1$ discontinuity but it can be used in different frameworks, in
    particular in problems with boundary conditions: we refer to Guerand \cite{G1} for results on
    state constraints problems and \cite{G2} in the case of Neumann conditions where ``effective
    boundary conditions and new comparison results are given, both works being in the case of
    quasi-convex Hamiltonians. 
\end{remark}

We give a first important property of \FLSub

\begin{proposition}\label{reg-sub-FL}\emph{--- Regularity of subsolutions.}\smsp
    Assume that \GAQC holds and that the Hamiltonian $G$ satisfies \GAGFL.
    Any \usc\ \FLSub is regular on $\H$.
\end{proposition}

\begin{proof}
    It is an immediate application of Proposition~\ref{reg-sub} since the Hamiltonian $\G$ defined
    for $x\in \R^N$, $t\in (0,\Tf)$, $r\in \R$, $(p,p_t)\in \R^{N+1}$ by $$ \G(x,t,r,(p,p_t)):= p_t
    + H_i(x,t,r,p) \; \hbox{if  }x \in \Omega_i, $$ $$ \G(x,t,r,(p,p_t)):= \max( p_t +
    H^+_1(x,t,r,p), p_t + H^-_2(x,t,r,p), p_t+G(x,t,r,p')) \; \hbox{if  }x \in \H, $$ satisfies the
    assumptions of this proposition with $y=(x',t)$, $z=x_N$, and in particular the normal
    controllability in the $x_N$-direction.  
\end{proof}

\section{Stability of flux-limited solutions}

In this section, we provide a result on the stability of flux-limited solutions. As the proof will
show it, such result is not an immediate extension of Theorem~\ref{hrl}; indeed, if the change of
test-functions does not really cause any problem, the formulation of flux-limited sub and
supersolutions with global Hamiltonians which are not \lsc or \usc is the source of difficulties.

The result is the\index{Flux-limited solutions!stability}

\begin{theorem}\label{FL-stab}\emph{--- Stability result for flux-limited solutions.}\smsp
    Assume that, for $\e >0$, $u_{\e}$ is a \FLSub \resp{\FLSup} for the problem with Hamiltonians
    $H_1^\e, H_2^\e, G^\e$.  We assume that $H_1^\e, H_2^\e, G^\e$ are continuous and $H_1^\e,
    H_2^\e$ satisfy \HQC.  If $H_1^\e, H_2^\e, G^\e$ converge locally uniformly to respectively
    $H_1, H_2, G$ and if the functions $u_{\e}$ are uniformly locally bounded on $ \R^N$, then
    $\ou=\limssup u_{\e}$ \resp{$\uu=\limiinf u_{\e}$} is a \FLSub \resp{\FLSup} for the problem
    with Hamiltonians $H_1, H_2, G$.  
\end{theorem}

\begin{proof}
    Due to the dissymmetry in the definitions of \FLSub and \FLSup, we have to give the proof in
    both cases.

    \noindent\textbf{(a)} We start by the \FLSub one. Of course, we have just to prove the result on
    $\H$ since, in $\Omega_1, \Omega_2$, the result is an easy application of Theorem~\ref{hrl}. Let
    $\psi=(\psi_1,\psi_2) \in \PC1$ and let $(x,t) \in \H \times (0,\Tf)$ be a strict local maximum
    point of $\ou -\psi$. We have to show that
    \[
    \max \Big(\psi_t + G(x,t,\ou,D_\H \psi), \psi_t+ H^+_1(x,t, \ou, D\psi_1) ,  
    \psi_t+ H^-_2(x,t, \ou, D\psi_2) \Big) \leq 0  \: .
    \]
    By Lemma~\ref{convptsmaxdisc}, there exists a subsequence $(x_{\e'},t_{\e'})$ of maximum point
    of  $u_{\e'}-\psi$ which converges to $(x,t)$ and such that $u_{\e'}(x_{\e'},t_{\e'})$ converges
    to $\ou(x,t)$. To get the $G$-inequality, we replace $\psi$ by $\psi+K| x_N|$. 
    Using the quasi-convexity property of $H_1$ and $H_2$, for $K$ large enough we get
    $$
        \psi_t+ H_1(x,t, \ou, D\psi_1) >0 \quad \hbox{and} \quad \psi_t+ H_2(x,t, \ou, D\psi_2)>0\; .
    $$
    Applying the result of Lemma~\ref{convptsmaxdisc} to this new $\psi$, we see that
    necessarily $x_{\e'} \in \H$. Then, passing to the limit in the \FLSub inequality for
    $(H^{\e'}_1)^+, (H^{\e'}_2)^-, G^{\e'}$, we end up with $\psi_t + G(x,t,\ou,D_\H \psi)\leq 0$
    since the term $K|x_N|$ does not affect $D_\H \psi$.

    It remains to prove the $H^+_1$ and $H^-_2$ inequalities and to do so, we come back to the
    original $\psi$. We assume that $\psi_t+ H^+_1(x,t, \ou, D\psi_1) >0$ and change $\psi$ into
    $\psi+K(x_N)_-$, for $K$ large enough.

    For $\e'$ small enough, $x_{\e'}$ cannot be in $ \Omega_1$: since $H_1\geq H^+_1$
    implies $\psi_t+ H_1(x,t, u, D\psi_1) >0$, hence the $H_1^{\e'}$ inequality cannot hold
    for $\e'$ small enough. Similarly, $x_{\e'}$ cannot be on $\H$ because of the $H^+_1$
    inequality. Finally $x_{\e'}$ cannot be in $\Omega_2$ for $K$ large enough, therefore we reach a
    contradiction which implies that $\psi_t+ H^+_1(x,t, \ou, D\psi_1) \leq0$.

    Arguing the same way for the case $\psi_t+ H^-_2(x,t, \ou, D\psi_2)>0$, the subsolution
    inequality is proved.

    \bigskip

    \noindent\textbf{(b)} For the \FLSup case, again we just have to treat the inequalities on $\H$
    and we assume that $(x,t) \in \H \times (0,\Tf)$ is a strict local minimum point of $\uu -\psi$
    where $\psi=(\psi_1,\psi_2) \in \PC1$. We have to show that 
    \[ 
        \max \Big(\psi_t + G(x,t,\uu,D_\H
        \psi), \psi_t+ H^+_1(x,t, \uu, D\psi_1) ,  \psi_t+ H^-_2(x,t,\uu, D\psi_2) \Big) \geq 0  \: .
    \]
    We argue by contradiction assuming that the three quantities in the $\max$ are strictly
    negative. Similarly to the \FLSub case, we claim that we can choose $K_1,K_2\geq 0$ such that
    $$
        \psi_t+ H_1(x,t, \uu, D\psi_1-K_1e_N) < 0 \quad \hbox{and} \quad 
        \psi_t+ H_2(x,t, \uu, D\psi_2+ K_2e_N)<0\;,
    $$
    which follows here also from the quasi-convexity of $H_1$ and $H_2$. To use
    it, we change $\psi$ in $\psi -K_1(x_N)_+-K_2(x_N)_-$ and notice that $(x,t)$ is still is a
    strict local minimum point of $\uu -\psi$ for this new $\psi$.

    Applying again Lemma~\ref{convptsmaxdisc}, there exists a subsequence $(x_{\e'},t_{\e'})$ of
    minimum points of  $u_{\e'}-\psi$  which converges to $(x,t)$ and such that
    $u_{\e'}(x_{\e'},t_{\e'})$ converges to $\uu(x,t)$. And we examine the possible inequalities for
    $(x_{\e'},t_{\e'})$. Clearly $x_{\e'}$ can be neither in $\Omega_1$ nor in $\Omega_2$ for $\e'$
    small enough because of the above property. 
    Hence $x_{\e'}\in \H$ and the \FLSup inequality holds for $(H^{\e'}_1)^+, (H^{\e'}_2)^-,
    G^{\e'}$. But passing to the limit as $\e'\to0$ in these inequalities yields a contradiction, so
    the supersolution inequality holds.
\end{proof}

\begin{remark}
    The main weakness of Theorem~\ref{FL-stab} is to be strictly restricted to the framework of
    flux-limited solutions for problems with quasi-convex Hamiltonians. Therefore it is not very
    flexible, in particular if we compare it with Theorem~\ref{JV-stab} in the case of junction
    viscosity solutions.
\end{remark}

\section{Comparison results for flux-limited solutions and applications}
\label{chap:FL}

This section is devoted to prove comparison results for flux-limited solutions; the original proofs
given in \cite{IM, IM-md} were based on the rather technical construction of a ``vertex function''.
We present here the simplified proof(s) of \cite{BBCI}.

\subsection{The convex case}

The main result here is the following. 
\index{Comparison result!for flux-limited solutions}
\index{Flux-limited solutions!comparison in the convex case}

\begin{theorem}\emph{--- Comparison principle, the convex case.} \label{comp-IM}\smsp
    Assume that either \GAConv or \GACC holds, that the Hamiltonian $G(x,t,r,p')$ is convex in
    $(r,p')$ and satisfies \GAGFL. If $u,v : \R^N \times (0,\Tf) \rightarrow \R$ are respectively an
    \usc bounded flux-limited subsolution and a \lsc bounded flux-limited supersolution of
    \HJgen-\FL and if $u(x,0) \leq v(x,0)$ in~$\R^N$,  then $u \leq v$ in $\R^N \times (0,\Tf)$.
\end{theorem}

\begin{proof} 
    In order to simplify the proof, we provide it only in the case when the Hamiltonians $H_1,H_2,G$
    are independent of $u$; the general case only contains minor additional technical difficulties.

    \bigskip

    \noindent\textbf{(a)} \emph{Reduction of the proof ---}
    First we follow Section~\ref{sect:htc} and  check \LOCa-evol: the function $\chi : \R^N
    \times (0,\Tf) \rightarrow \R$ defined by $$\bar
    \chi(x,t):=-Kt-(1+|x|^2)^{1/2}-\frac{1}{\Tf-t}\; ,$$ is, for $K>0$ large enough, a strict
    subsolution of \HJgen-\FL with $\bar \chi(x,t) \to -\infty$ when $|x|\to +\infty$ or $t\to
    \Tf^-$. We replace $u$ by either $u_\mu:=u + (1-\mu) \bar \chi $ (a choice which does not use
    the convexity of the Hamiltonians) or $u_\mu:=\mu u + (1-\mu) \bar \chi $ (a choice which uses
    the convexity of the Hamiltonians). Borrowing also the arguments of Section~\ref{sect:htc},
    \LOCb-evol also holds and therefore we are led to show that \LCR-evol is valid in the case when
    $u$ is an $\eta$-strict subsolution of \HJgen-\FL.

    For a point $(\xb,\tb)$ where $\xb\in \Omega_1$ or $\xb\in \Omega_2$, the proof of \LCR-evol in
    $\overline{Q^{\xb,\tb}_{r,h}}$ is standard, hence we have just to treat the case when $\xb \in
    \H$. At this point, we make an other reduction in the proof: using Section~\ref{sect:sup.reg},
    with $y=(t,x')$ and $z=x_N$, since  \GAConv or \GACC are nothing but Assumptions \TC,\NCe and
    \Mong, Theorem~\ref{reg-by-sc} applies. As a consequence, we can assume \wlg that $u$ is
    Lipschitz continuous with respect to all variables and semi-convex in the $(t,x')$-variables.
    But we may also use the ideas of Proposition~\ref{C1-reg-by-sc} to obtain a subsolution which is
    $C^1$ in $(t,x')$ with $u_t$ and $D_{x'}u$ continuous \wrt all variables: indeed, we can apply
    the ideas of the proof of Proposition~\ref{C1-reg-by-sc} separately in $\Omega_1$, $\Omega_2$
    and $\H$ to obtain the $H_1$, $H_2$ and $G$ inequalities for the regularized function, while the
    $H_1^+$ and $H_2^-$ ones are deduced from Proposition~\ref{sub-up-to-b}.

    Then we assume that
    $$\displaystyle M:= \max_{\overline{Q^{x,t}_{r,h}}}(u-v)>0\;.$$ If
    this maximum is achieved on $\partial_p Q^{x,t}_{r,h}$, the result is obvious so we may assume
    that it is achieved at $(\tilde x, \tilde t) \notin \partial_p Q^{x,t}_{r,h}$.  Again, if
    $\tilde x \in \Omega_1$ or $\tilde x \in \Omega_2$, we easily obtain a contradiction and
    therefore we can assume that $\tilde x \in \H$.
  
    \bigskip

    \noindent\textbf{(b)} \emph{Building the test function ---}
    Setting $a=u_t (\tilde x, \tilde t)$, $p'=D_{x'} u (\tilde x, \tilde t)$, we claim that we can
    solve the equations 
    $$ a + H_{1}^- (\tilde x, \tilde t,\tilde p'+\lambda_1 e_N)= -\eta/2
    \quad,\quad  a + H_{2}^+ (\tilde x, \tilde t,\tilde p'+\lambda_2 e_N)=-\eta/2
     \;,$$
    where we recall that $-\eta$ is the constant which measures the strict subsolution property of
    function $u$.

    In order to prove the existence of $\lambda_1$, we look at maximum points of
    $$ u(x,t) - \frac{|x-\tilde x|^2}{\e^2}-\frac{|t-\tilde t|^2}{\e^2} - \frac{\e}{x_N}$$ in
    $(\overline{Q^{x,t}_{r,h}})\cap (\Omega_1\times :[0,\Tf])$, and for $0<\e\ll 1$. This function
    achieves its maximum at $(\xe,\te)$ which converges to $(\tilde x, \tilde t)$ as $\e \to 0$ and
    by the semi-convexity of $u$ in $t$ and $x'$, one has $$ u_t (\xe,\te)+ H_1(\xe,\te, D_{x'} u
    (\xe,\te) + \lambda_\e e_N) \leq -\eta\;,$$ for some $\lambda_\e \in \R$. Moreover, $\lambda_\e$ 
    is bounded \wrt $\e$ since $u$ is Lipschitz continuous. 
    Letting $\e$ tend to $0$ and using that $u_t (\xe,\te)\to a, D_{x'}
    u (\xe,\te)\to p'$ by the semi-convexity property of $u$, together with the extraction of a
    subsequence for $(\lambda_\e)_\e$, we get a $\bar \lambda\in \R$ such that 
    $$ a + H_{1} (\tilde x, \tilde t,\tilde p'+\bar\lambda e_N)\leq  -\eta\;.$$ 
    Since $H_{1}^- \leq H_1$, it follows that 
    $a + H_{1}^- (\tilde x, \tilde t,\tilde p'+\bar\lambda e_N)\leq  -\eta$. Then we use the
    fact that $\lambda \mapsto a + H_{1}^- (\tilde x, \tilde t,\tilde p'+\lambda e_N)$ is
    continuous, nondecreasing on $\R$ and tends to $+\infty$ when $\lambda\to +\infty$ to get the
    existence of $\lambda_1>\bar\lambda$ solving the equation with $-\eta/2$. In this framework,
    $\lambda_1$ is necessarily unique since the convex function $\lambda \mapsto a + H_{1}^- (\tilde
    x, \tilde t,\tilde p'+\lambda e_N)$ only has flat parts at its minimum, while clearly
    $\lambda_1$ is not a minimum point for this function.  
    The proof for $\lambda_2$ is analogous and we skip it.

    In order to build the test-function, we set, for $z \in \R$, $h(z):=\lambda_1z_+-\lambda_2z_-$ 
    where $z_+=\max(z,0)$, $z_-=\max(-z,0)$, and
    \begin{equation} \label{defchila}
        \chi(x_N,y_N): = h(x_N)-h(y_N)=
        \apg
        \begin{array}{ll}
        \lambda_1 ( x_N - y_N) & \mbox{ if }  x_N \geq 0 \: , \: y_N \geq 0\;, \\
        \lambda_1  x_N-\lambda_2
        y_N & \mbox{ if }  x_N \geq 0 \: , \: y_N < 0\;, \\
        \lambda_2 x_N - \lambda_1
        y_N &  \mbox{ if }  x_N < 0 \: , \: y_N \geq 0\;, \\
        \lambda_2( x_N - y_N ) & \mbox{ if }  x_N < 0 \: , \: y_N < 0 \;.
        \end{array} \ch
    \end{equation}
    Then, for $0< \e \ll 1$ we define a test function as follows
    $$
    \psi_{\e}(x,t,y,s):= \frac{|x-y|^2}{\e^2} + \frac{|t-s|^2 }{\e^2}+\chi(x_N,y_N)+
    |x-\tilde x|^2+|t-\tilde t|^2\; .$$
    In view of the definition of $h$, we see that for any $(x,t) \in \R^N\times [0,\Tf]$ the
    function $\psi_{\e}(x,\cdot,t,\cdot) \in \PC1$ and for any $(y,s) \in \R^N\times [0,\Tf]$ the
    function $\psi_{\e}(\cdot,y,\cdot,s) \in \PC1$.

    We now look at the maximum points of
    $$ (x,t,y,s) \mapsto u(x,t)-v(y,t)-\psi_{\e}(x,t,y,s)\quad\text{in}\quad
    \left[\overline{Q^{x,t}_{r,h}}\right]^2\;.$$
    By standard arguments, this function has maximum points
    $(\xe,\te,\ye,\se)$ such that $(\xe,\te,\ye,\se)\to (\tilde x,\tilde t,\tilde x,\tilde t)$. 
    Moreover, using the semi-convexity of $u$, we have
    $$ p'_\e=\frac{2(\xe'-\ye')}{\e^2} \to p' \quad\hbox{and} 
    \quad \frac{2(\te-\se) }{\e^2}\to a\; ,$$
    which the Lipschitz continuity of $u$ implies that
    $ (p_\e)_N=2((\xe)_N-(\ye)_N)/\e^2$  remains bounded.

    \bigskip

    \noindent\textbf{(c)} \emph{Getting contradictions ---}
    We have to consider different cases depending on the position of $\xe$ and $\ye$ in $\R^N$. Of
    course, we have no difficulty for the cases $\xe,\ye \in \Omega_1$ or $\xe,\ye \in \Omega_2$,
    and even less because of the above very precise properties on the derivatives of the
    test-function; only the cases where $\xe$, $\ye$ are in different domains or on $\H$ cause
    problem. So, we are left with considering three cases
    \begin{enumerate}
        \item $\xe \in \Omega_1$, $\ye \in \Omegb_2$ or 
            $\xe \in \Omega_2$, $\ye \in \Omegb_1$.
    \item $\xe \in \H$, $\ye \in (\Omega_1\cup \Omega_2)$.
    \item  $\xe \in \H$, $\ye \in \H$.
    \end{enumerate}

    \noindent\textbf{Case $\mathbf{1}$:}  
    If $\xe \in \Omega_1$, $\ye \in \Omega_2 \cup \H$, we use that $u$ is an $\eta$-strict
    $H_1$-subsolution and taking into account the specific form of the test-function above
    we get
    \begin{equation}\label{ineq:subsol}
    a+o_\e(1)+  H_1(\xe,\te, p' + o(1) + \lambda_1 e_N + (p_\e)_N e_N) \leq -\eta\;.
    \end{equation}
    Then, using that $H_1\geq H_1^-$ and the fact that every term in $H_1$ remains in a compact
    subset, we also have
    $$a+  H_1^-(\tilde x,\tilde t, p' + \lambda_1 e_N + (p_\e)_N e_N) \leq -\eta +o_\e(1)\; .$$
    Now, since $(p_\e)_N\geq 0$, thanks to the monotonicity of $H_1^-$ in the $e_N$-direction
    we obtain
    $$a+  H_1^-(\tilde x,\tilde t, p' + \lambda_1 e_N) \leq -\eta +o_\e(1)\; ,$$
    which is a contradiction with the definition of $\lambda_1$. The case $\xe \in \Omega_2$, $\ye
    \in \Omega_1 \cup \H$ is of course analogue and we skip it.

    \bigskip

    \noindent\textbf{Case $\mathbf{2}$:} Since $\xe \in \H$, the subsolution inequality holds
    \[\begin{aligned}
    \max \Big(  a+ G(\tilde x,\tilde t, p') \ ;\  &
    a+  H_1^+ (\tilde x,\tilde t, p' + \lambda_1 e_N + (p_\e)_N e_N)  \  ;\ \\
    & a+ H_2^-(\tilde x,\tilde t, p' + \lambda_1 e_N + (p_\e)_N e_N)  \Big) \leq -\eta +o_\e(1)\; .
    \end{aligned}
    \]
    On the other hand, if $\ye \in \Omega_1$, since $v$ is a $H_1$-supersolution in $\Omega_1$ and
    of course $(y_\e,t_\e)\to(\tilde x,\tilde t)$,
    \be \label{dis1}
    a+ H_1(\tilde x,\tilde t, p' + \lambda_1 e_N + (p_\e)_N e_N) \geq o_\e(1)\;.
    \ee
    Now the aim is to show that the same inequality holds for $H_1^+$ and to do so, we evaluate this
    quantity for $H_1^-$: taking into account the fact that here $(p_\e)_N \leq 0$, the monotonicity
    of $H_{1}^-$ in the $e_N$-direction yields 
    $$a+ H_1^- (\tilde x,\tilde t, p' + \lambda_1 e_N + (p_\e)_N e_N) \leq -\eta/2 + O_\e(1)< 0
    \;\hbox{if $\e$ is small enough.}$$
    But since $H_1=\max(H_{1}^-,H_{1}^+)$, from \eqref{dis1} we actually deduce that
    $$a+ H_1^+(\tilde x,\tilde t, p' + \lambda_1 e_N + (p_\e)_N e_N)\geq o_\e(1)\;,$$
    which gives a contradiction when compared with the subsolution property on $\H$. The same
    contradiction is obtained in the case $\ye \in \Omega_2$, using $\lambda_2$ and $H_2^+$ instead
    of $\lambda_1$ and $H_1^-$.

    \bigskip

    \noindent\textbf{Case $\mathbf{3}$:} If $\xe \in \H$, $\ye \in \H$, we have viscosity sub and
    supersolution inequalities for the same Hamiltonian and the contradiction follows easily. So,
    the proof is complete.  
\end{proof}

\subsection{The quasi-convex case }
\label{subsec:gen.QC}

    In fact, Theorem~\ref{comp-IM} extends without difficulties in the ``quasi-convex'' case and we
    have the\index{Flux-limited solutions!comparison in the quasi-convex case}
    
\index{Comparison result!for flux-limited solutions (quasi-convex case)}
\begin{theorem}\label{comp-IM-nc}\emph{--- Comparison principle, the quasi-convex case.}\smsp
    The result of Theorem~\ref{comp-IM} remains valid if \GAQC holds and $G$ satisfies \GAGFL.
\end{theorem}

\begin{proof}
    We just sketch it since it follows very closely the proof of Theorem~\ref{comp-IM}.  The only
    difference here is that Section~\ref{sect:sup.reg} only allows to reduce to the case when the
    strict subsolution $u$ is Lipschitz continuous and semi-convex in the $(t,x')$-variables, not
    $C^1$. This obliges us to first look at a maximum of 
    $$ (x,t,y,s) \mapsto u(x,t) -v (y,s) -\frac{|x'-y'|^2}{\e^2} - \frac{|t-s|^2 }{\e^2}\; ,$$
    where $x=(x',x_N)$, $y=(y',x_N)$, which is, of course, an approximation of 
    $\displaystyle \max_{\overline{Q^{x,t}_{r,h}}}\,(u-v)$.

    If $(\tilde x, \tilde t,\tilde y, \tilde s)$ is a maximum point of this function, 
    the semi-convexity of $u$ implies that $u$ is differentiable \wrt $x'$ and $t$ at $(\tilde x,
    \tilde t)$ and we have
    $$ a:=\frac{2(\tilde t-\tilde s)}{\e^2}=u_t (\tilde x, \tilde t)
    \quad\hbox{and}\quad p':=\frac{2(\tilde x'-\tilde y')}{\e^2}=D_{x'} u (\tilde x, \tilde t)\;.$$
    Then we solve the $(\lambda_1,\lambda_2)$-equations with such $a$ and $p'$; it is worth pointing
    out that $\lambda_1$ and $\lambda_2$ are not uniquely defined but this is not important in the
    proof.

    Finally we consider the maxima of the function
    $$\begin{aligned}
        (x,t,y,s) \mapsto u(x,t) -v (y,s) &  -\frac{|x'-y'|^2}{\e^2} - \frac{|t-s|^2 }{\e^2}\\
        & -\chi(x_N,y_N)-\frac{|x_N-y_N|^2}{\gamma^2} - |x-\tilde x|^2 - |t-\tilde t|^2\; ,
    \end{aligned}$$
    where $0<\gamma\ll 1$ is a parameter devoted to tend to $0$ first.
    Using the normal controllability assumption with variables $X=(x',t)$, $Z=x_N$, it is easy to show that 
    $$ |(p_\e)_N|=\frac{2|(\xe)_N-(\ye)_N|}{\gamma^2}=O(|p'_\e|+|a|+1)\;,$$
    which is bounded since $u$ is Lipschitz continuous in the tangent variables $(x',t)$.
    This allows to perform all the arguments of the proof as in the convex case.
    Notice that, even if it is not $C^1$-smooth, the semi-convexity of $u$ ensures that $u_t
    (\xe,\te)\to a, D_{x'} u(\xe,\te)\to p'$.  
\end{proof}

\section{Flux-limited solutions and control problems}
\label{sect:control.NA}

In this section, we come back on the control problem of Section~\ref{sect:codimIa} which we address
here from a different point of view.\index{Control problem!with flux~limiters}

In order to do that, we first have to define the admissible trajectories among all the solutions of
the differential inclusion: we say that a solution $(X,D,L)(\cdot)$ of the differential inclusion
starting from $(x,t,0,0)$ is an admissible trajectory if
\begin{enumerate}
\item there exists a global control $a=(\alpha_1,\alpha_2,\alpha_0)$
  with $\alpha_i\in \mathcal{A}_i:=L^\infty(0,\infty;A_i)$ for
  $i=0,1,2$;
\item there exists a partition $\I=(\I_1, \I_2, \I_0)$ of
  $(0,+\infty)$, where $\I_1, \I_2, \I_0$ are measurable sets, such
  that $X(s)\in\overline{{\Omega}_i}$ for any $s\in \I_i$ if $i=1,2$ and
  $X(s)\in\H$ if $s\in \I_0$;
\item for almost every $0\leq s\leq t$ 
\be\label{eqn:admtraj}
(\dot{X},\dot{D},\dot{L})(s)=\sum_{i=0}^2 (b_i ,c_i,l_i) (X(s),t-s,\alpha_i(s))\1_{\I_i} (s)\;.  
\ee
\end{enumerate}

In Equation~\eqref{eqn:admtraj}, we have dropped $T(s)$ since we are in the $b^t\equiv -1$ case and
therefore $T(s)= t-s$ for $s\leq t$. The set of all admissible trajectories $(X,\I,a)$ issued from a
point $X(0)=x\in\R^N$ (at $T(s)=t)$ is denoted by $\mT_x$. Notice that, under the controllability
assumption \NCoH, for any point $x\in\Omegb_1$, there exist trajectories starting from
$x$, which stay in $\Omegb_1$, and the same remark holds for points in
$\Omegb_2$. These trajectories are clearly admissible (with either $\I_1\equiv \I$ or
$\I_2\equiv \I$) and therefore $\mT_x$ is never void. 

\begin{remark}
    It is worth pointing out that, in this approach, the partition $\I_0, \I_1, \I_2$ which we
    impose for admissible trajectories, implies that there is no mixing on $\H$ between the dynamics
    and costs in $\Omega_1$ and $\Omega_2$, contrarily to the approach of
    Section~\ref{sect:codimIa}. A priori, on $\H$, either we have an independent control problem or
    we can use either $(b_1, c_1,l_1)$ or $(b_2, c_2,l_2)$, but no combination of $(b_1,c_1,l_1)$
    and $(b_2,c_2,l_2)$.
\end{remark}

The value function is then defined as 
$$  \Uim_G(x,t):=\inf_{(X,\I,a)\in\mT_x} \left\{ \int_0^{t} 
    \bigg(  \sum_{i=0}^2 l_i (X(s),t-s,\alpha_i(s))\1_{\I_i}(s)\bigg)
    e^{-D(s)}\d s + u_0(X(t)) \right\}\;,
$$
where $u_0\in C(\R^N)$.

As always, the first key ingredient to go further is the
\index{Dynamic Programming Principle!for flux-limited value functions}
\begin{lemma}\emph{--- Dynamic Programming Principle.}\smsp
    Under assumption \GACC, the value function $\Uim_G$ satisfies: for all $(x,t)\in \R^N \times
    (0,\Tf]$ and $\tau<t$ 
    $$  \Uim_G(x,t)=\inf_{(X,\I,a)\in\mT_x} \left\{ \int_0^{\tau} 
    \bigg(  \sum_{i=0}^2 l_i (X(s),t-s,\alpha_i(s))\1_{\I_i}(s)\bigg)
    e^{-D(s)}\d s + \Uim_G(X(\tau),t-\tau) \right\}\; .
    $$
\end{lemma}
We leave the easy proof of this lemma to the reader, which is standard. Now, using standard
arguments based on the Dynamic Programming Principle and the comparison result, we have
the\index{Flux-limited solutions!and control problems}
\begin{theorem}\label{thm:cont-FL}
    Under assumption \GACC and if $u_0\in C(\R^N)$, the value function $\Uim_G$ is the unique
    flux-limited solution of \HJgen-\FL with $G=H_0$ given by
     $$ H_0(x,t,r,p)= \sup_{\alpha_0\in A_0} 
    \apg  -b_0(x,t,\alpha_i) \cdot p +c_0(x,t,\alpha_i)r - l_0(x,t,\alpha_i) \chg\;.$$
\end{theorem}

\begin{proof}
We describe some non-obvious parts of the proof, in particular those to show that the value function
$\Uim_G$ is a flux-limited solution of \HJgen-\FL.  As we will explain at the end of the proof,
continuity of $\Uim_G$ and its uniqueness are an immediate consequence of Theorem~\ref{comp-IM}.

\medskip

    \noindent\textbf{(a)} \emph{Subsolution property.} 

    Of course, the only difficulty is to prove this property on $\H \times (0,\Tf]$, the cases of
    $\Omega_1 \times (0,\Tf]$ and $\Omega_2 \times (0,\Tf]$ being classical. To do so, we have to
    show that
    \begin{equation}\label{ineq:FL.sub}
    ( \Uim_G)^*_t-b_i(x,t,\alpha_i)\cdot D (\Uim_G)^*+c_i(x,t,\alpha_i)(\Uim_G)^*-l_i(x,t,\alpha_i)\leq 0\; ,
    \end{equation}
    for any $i=0,1,2$ any $\alpha_i \in A_i$ with $b_1(x,t,\alpha_i)\cdot e_N \geq 0$ if $i=1$ and
    $b_2(x,t,\alpha_i)\cdot e_N \leq 0$ if $i=2$. The proof of these inequalities is standard
    once we use the following two remarks: 
    \begin{enumerate}
    \item[$1.$] By the arguments of Theorem~\ref{SubP} which give such result in a more general
    setting, if $\Uim_G(\xe,\te)\to ( \Uim_G)^*(x,t)$, we can assume without loss of generality that
    $(\xe,\te) \in \H\times (0,\Tf]$. This first remark allows to prove \eqref{ineq:FL.sub} in the case
    $i=0$ using classical arguments.  
    \item[$2.$] The convexity of
    $\BCL_1(x,t)=\{(b_1(x,t,\alpha_1),c_1(x,t,\alpha_1),l_1(x,t,\alpha_1)):\ \alpha_1 \in A_1\}$
    together with the normal controllability assumption implies that the set
    $$\big\{(b_1(x,t,\alpha_1),c_1(x,t,\alpha_1),l_1(x,t,\alpha_1)):
    \ b_1(x,t,\alpha_1)\cdot e_N \geq 0, \ \alpha_1 \in A_1\big\}$$
    is the closure of the set
    $$ \big\{(b_1(x,t,\alpha_1),c_1(x,t,\alpha_1),l_1(x,t,\alpha_1)):
    \ b_1(x,t,\alpha_1)\cdot e_N > 0, \ \alpha_1 \in A_1\big\}\;,$$
    and an analogous property holds for $i=2$. This remark reduces the proof of 
    \eqref{ineq:FL.sub} for $\alpha_1$ and $\alpha_2$ such that $b_1(x,t,\alpha_1)\cdot e_N > 0$ and
    $b_2(x,t,\alpha_2)\cdot e_N < 0$. And this allows to use classical arguments since,
    for $s\in (0,\tau]$ and $\tau$ small enough, trajectories $X(s)$ which are associated to
    such dynamics with constant controls remains in $\Omega_1$ in the first case and in
    $\Omega_2$ in the second one.
\end{enumerate}
We point out that Property $(a)$ plays a key role to obtain the three types of inequalities for
$i=0,1,2$.

\medskip

    \noindent\textbf{(b)} \emph{Supersolution property.}

    Again the only non-classical case concerns points of $\H \times (0,\Tf]$.  Let $(x,t) \in \H
    \times (0,\Tf]$ be a minimum point of $(\Uim_G)_*-\phi$ where $\phi=(\phi_1,\phi_2) \in \PC1$.
    We assume \wlg that $(\Uim_G)_*(x,t)=\phi(x,t)$.

    We first fix $0<\tau \ll 1$ and, for $0<\e \ll 1$, we consider $(\xe,\te)$ such that $\Uim_G
    (\xe,\te) \leq (\Uim_G)_*(x,t) +\e \tau$ with $|(\xe,\te)-(x,t)|\leq \e\tau$. Then we choose a
    global $\e$-optimal control $a^\e=(\alpha_1^\e,\alpha_2^\e,\alpha_0^\e)$ and denote by
    $Z_i^\e=Z_i^\e(s)=\big(X^\e(s),\te-s,\alpha^\e_i(s)\big)$ for simplicity of notations. In other words,
    $$
    \Uim_G(\xe,\te)\geq \int_0^{\tau} \bigg(  \sum_{i=0}^2 l_i (Z_i^\e)\1_{\I_i}(s)\bigg)
    e^{-D^\e(s)}\d s + \Uim_G(X^\e (\tau),\te-\tau)- \e \tau\; ,
    $$
    where $X^\e,D^\e$ are the trajectory and the discount term computed with the global control
    $a^\e$.  Using the minimum point property, we have
    $$
    \phi (\xe,\te)\geq \int_0^{\tau} \bigg(  \sum_{i=0}^2 l_i (Z_i^\e)\1_{\I_i}(s)\bigg)
    e^{-D^\e(s)}\d s + \phi(X^\e (\tau),\te-\tau)- 2\e \tau\; ,
    $$
    and by classical computations we obtain
    $$\begin{aligned}
        \int_0^{\tau}  \sum_{i=0}^2 \bigg( & (\phi_i)_t (X^\e(s),\te-s)-b_i(Z_i^\e)\cdot D\phi_i
        (X^\e(s),\te-s)\\
        & + c_i(Z_i^\e)\phi_i (X^\e(s),\te-s)-l_i(Z_i^\e)\bigg)  \1_{\I_i}(s)
    e^{-D^\e(s)}\d s \geq - 2\e \tau\;,
    \end{aligned}$$
    where, by convention, $\phi_0$ denotes $\phi_1=\phi_2$ on $\H\times (0, \Tf]$.

    Then, by using the regularity of $\phi_i$ ($i=1,2$), 
    $$\begin{aligned}
    \int_0^{\tau}  \sum_{i=0}^2 \bigg( & (\phi_i)_t (\xe,\te)-b_i(Z_i^\e)\cdot D\phi_i (\xe,\te)\\ 
    & + c_i(Z_i^\e)\phi_i (\xe,\te)-l_i(Z_i^\e)\bigg)  \1_{\I_i}(s)
    e^{-D^\e(s)}\d s \geq - 2\e \tau+o(\tau)\;.
    \end{aligned}$$
    
    In order to conclude, we have to consider several cases
    \begin{enumerate}
    \item[$(i)$] If $\I_0=(0,\tau)$, the proof just follows classical arguments.
    \item[$(ii)$] If $\I_1=(0,\tau)$, \ie the trajectory $X^\e$ remains in $\Omegb_1$, we notice that
        $$ \frac 1 \tau \int_0^{\tau} b_1(Z_i^\e)ds\cdot e_N = 
        \frac 1 \tau ( X^\e (\tau)- \xe)\cdot e_N \geq -\e \; ,$$
         because of the choice of $(\xe,\te)$. Using the convexity and the compactness of
        $\BCL_1(x,t)$, we conclude that as $\tau,\e\to 0$, up to the extraction 
        of a subsequence, we may assume that
        $$ \frac 1 \tau \int_0^{\tau} \big(b_1(Z_1^\e),c_1(Z_1^\e),l_1(Z_1^\e)\big) \ds \to
        \big(b_1(x,t,\bar\alpha_1), c_1(x,t,\bar \alpha_1),l_1(x,t,\bar \alpha_1)\big)$$ 
        for some $\bar\alpha_1\in A_1$ such that $b_1(x,t,\bar \alpha_1)\cdot e_N\geq 0$. 
        From there, one concludes easily that the $H_1^+$-term is non-negative.
    \item[$(iii)$] If $\I_2=(0,\tau)$, the same arguments allow to conclude that the $H_2^-$-term is
        non-negative. 
    \item[$(iv)$] The remaining case is when two of these three sets are non-empty, and the main
        difficulty is when one of the open sets (or both) $\{s:X^\e(s)\in \Omega_i\}$ is non-empty. We
        assume, for example, that it is the case for $i=1$ and write 
        $$\{s:X^\e(s)\in \Omega_1\}=\bigcup_k\,  ]s_k,s_{k+1}[\; .$$ 
        If $s_k > 0$ and $s_{k+1}> \tau$, we necessarily $X(s_k)\in \H$ and $X(s_{k+1})\in \H$, therefore
        $$ \frac 1{s_{k+1}-s_k} \int_{s_k}^{s_{k+1}} b_1(Z_1^\e)\ds\cdot e_N = 
        \frac 1 \tau ( X^\e (s_{k+1})- X^\e (s_k))\cdot e_N=0 \; .$$
        Using again the convexity and the compactness of $\BCL_1(x,t)$, together
        with the regularity properties of $b_1,c_1,l_1$, we deduce that
        $$\begin{aligned}
         \frac 1{s_{k+1}-s_k} & \int_{s_k}^{s_{k+1}} 
        \bigg(  (\phi_1)_t (\xe,\te)-b_1(Z_1^\e)\cdot D\phi_i (\xe,\te)\\ 
        & + c_1(Z_1^\e)\phi_i (\xe,\te)-l_1(Z_i^\e)\bigg) e^{-D^\e(s)}\ds \\
        & \leq (\phi_1)_t (x,t) +H_1^+(x,t,\phi_1(x,t), D\phi_1(x,t)) + 2\e \tau+o(\tau)\;.
        \end{aligned}$$
        To obtain this last inequality, we have used that if
        $$ H_{1,\eta}^+(x,t,r,p):=\sup_{\alpha_i\in A_{1,\eta}^+} 
        \apg  -b_i(x,t,\alpha_i) \cdot p +c_i(x,t,\alpha_i)r - l_i(x,t,\alpha_i) \chg\;,
        $$
        where  $A_{1,\eta}^+:=\{ \alpha_1 \in A_1\:: \: b_1(x,t,\alpha_i) \cdot e_N  \geq  \eta\}$ and
        $\eta$ can be positive or negative, then $H_{1,\eta}^+(x,t,r,p)\to H_1^+(x,t,r,p)$ locally
        uniformly when $\eta \to 0$, a property which can be easily proved using the normal
        controllability.

        Using similar ideas, one can easily treat the cases $s_k = 0$ or $s_{k+1}= \tau$ and, of
        course, the case when $\{s:X^\e(s)\in \Omega_2\}$ is not empty.  Gathering all these
        informations, we end up showing that a convex combination of $\phi_t +
        H_1^+, \phi_t + H_2^-, \phi_t + H_0$ is non-negative, hence the result.
 \end{enumerate}

 \medskip

 \noindent\textbf{(c)} \emph{Continuity and uniqueness.}

    The function $\Uim_G$ being a discontinuous flux-limited solution of \HJgen-\FL,
    Theorem~\ref{comp-IM} shows that $(\Uim_G)^*\leq (\Uim_G)_*$ in $\R^N\times [0,\Tf]$; indeed it
    is easy to show that $(\Uim_G)^*(x,0)=(\Uim_G)_*(x,0)=u_0(x)$ in $\R^N$. Therefore $\Uim_G$ is
    continuous and the uniqueness comes from the same comparison result.
\end{proof}

Before considering the connections with the results of Section~\ref{sect:codimIa}, we want to point
out that among all these ``flux-limited value functions'', there is a particular one which
corresponds to either no specific control on $\H$ (\ie we just consider the trajectories such that
$\I_0\equiv \emptyset$) or, and this is of course equivalent, to a cost $l_0=+\infty$.
This value function is denoted by $\Uim$.

The aim is to show that the value functions of regional control are flux-limited
solutions.\index{Flux-limited solutions!and control problems}

\begin{theorem}\label{compBBC-IM}\emph{--- Identification of extremal Ishii solutions.}\smsp
Under the assumptions of Theorem~\ref{comp-IM} (comparison result), for any Hamiltonian $H_0$ we have
\begin{enumerate}
\item[$(i)$] $\Um \leq \Up \leq \Uim$ in $\R^N \times [0,\Tf]$.
\item[$(ii)$] $\Um = \Uim_{G}$ in $\R^N\times [0,\Tf]$ where $G=\HT$ and
    $\Um_{H_0} = \Uim_{G}$ in $\R^N\times [0,\Tf]$ where $G=\max(\HT,H_0)$.
\item[$(iii)$] $\Up = \Uim_{G}$ in $\R^N\times [0,\Tf]$ where $G=\HT^{\rm reg}$.
\end{enumerate}
\end{theorem}

This result shows that, by varying the flux~limiter $G$, we have access to the different
value functions described in Section~\ref{sect:codimIa}.

\begin{proof} 
    For $(i)$, the inequalities can just be seen as a consequence of the definition of $\Um ,\Up ,
    \Uim$ remarking that we have a larger set of dynamics-costs for $\Um$ and $\Up$ than for $\Uim$.
    From a more pde point of view, applying  Proposition~\ref{sub-up-to-b}, it is easy to see that
    $\Um ,\Up$ are flux-limited subsolutions of (HJ-gen)-(FL) since they are of course subsolutions
    of
    $$ u_t+ H_1^+(x,t,u, Du) \leq 0 \quad\hbox{in  }\Omega_1\times [0,\Tf] \; ,$$
    $$ u_t+ H_2^-(x, t,u,Du) \leq 0 \quad\hbox{in  }\Omega_2 \times [0,\Tf]\; .$$
    Then Theorem~\ref{comp-IM} allows us to conclude.

    For $(ii)$ and $(iii)$, we have to prove respectively that $\Um$ is a solution of (HJ-gen)-(FL)
    with $G=\HT$, $\Um_{H_0}$ is a solution of (HJ-gen)-(FL) with $G=\max(\HT,H_0)$ and $\Up$ with
    $G=\HT^{\rm reg}$. Then the equality is just a consequence of Theorem~\ref{comp-IM}.

    For $\Um$, the subsolution property just comes from the above argument for the $H_1^+,
    H_2^-$-inequalities and from Proposition~\ref{prop:complemented.one-d} for the $\HT$-one. The
    supersolution inequality is a consequence of the proof of Lemma~\ref{lem:comp.fundamental}:
    alternative {\bf A)} implies that one of the $H_1^+, H_2^-$-inequalities hold while alternative
    {\bf B)} implies that the $\HT$-one holds. The same is true for $\Um_{H_0}$.

    For $\Up$, the subsolution property follows from the same arguments as for $\Um$, both for the
    $H_1^+, H_2^-$-inequalities and from Proposition~\ref{prop:ishii-Up} for the $\HT^{\rm
    reg}$-one. The supersolution inequality is a consequence of Theorem~\ref{teo:condplus.VFp}:
    alternative {\bf A)} implies that one of the $H_1^+, H_2^-$-inequalities hold while alternative
    {\bf B)} implies that the $\HTreg$-one holds.

    And the proof is complete.
\end{proof}

\bigskip

Notice that inequalities in Theorem~\ref{compBBC-IM}-$(i)$ can be strict: various
examples are given in \cite{BBC1}. The following one shows that we can have $\Up < \Uim$ in~$\R$.
\begin{example}  Let
    $\Omega_1=(0,+\infty)$, $\Omega_2=(-\infty, 0)$. We choose $c\equiv 0$, $u_0(x)=0$ in $\R$ and
    \[
    b_1(\alpha_1)=\alpha_1 \in [-1,1]\; , \; l_1(\alpha_1)= \alpha_1\; ,
    \]
    \[
    b_2(\alpha_2)=\alpha_2 \in [-1,1]\; , \; l_1(\alpha_2)= -\alpha_2\; .
    \]
    It is clear that the best strategy---\ie with the minimal cost---is to use $\alpha_1=-1$ in
    $\Omega_1$, $\alpha_2=1$ in $\Omega_2$. We can also use these strategies at $0$ since
    \[ \frac12  b_1(\alpha_1) +\frac12  b_2(\alpha_2) = 0\; ,\] a combination which yields a cost of
    $-1$.  Therfore, an easy computation gives 
    \[ \VFp(x,t) = \int_0^{t} -1.dt = -t\; ,\]
    other words, the ``push-push'' strategy at $0$ allows to maintain the $-1$ cost.

    But, for $\Uim$, this ``push-push'' strategy at $0$ is not allowed and, since the optimal
    trajectories are necessarily monotone, the best strategy when starting at $0$ is to stay at $0$.
    Here, the best possible cost is $0$. 

    Hence $\Uim(0,t) = 0 > \VFp(0,t)=-1$, and in fact it can be shown that
    $$\Uim(x,t)=-|x| > \Up(x,t)=-t\quad\emph{if}\quad |x|<t\;.$$
    On the contrary, for $|x|\geq t$, $\Uim(x,t)= \Up(x,t)=-t$ since the above strategy with
    $\alpha_1=-1$ in $\Omega_1$, $\alpha_2=1$ in $\Omega_2$ can be applied for all time.
\end{example}

Theorem~\ref{compBBC-IM} can be interpreted in several ways but the key point is to chose the kind
of controlled trajectories we wish to allow on $\H$. Then, depending on this choice, different
formulations have to be used for the associated HJB problem. It could be thought that the
flux-limited approach is more appropriate, in particular because of Theorem~\ref{comp-IM} which is
used intensively in the above proof.

\section{Vanishing viscosity approximation (I): convergence via flux-limited solutions}
\label{sect:vanishing}

\index{Vanishing viscosity method!via flux-limited solutions}

In the framework of classical viscosity solutions, getting the convergence of the vanishing viscosity
method is just a simple exercice done either with a stability result, or the combination of the
half-relaxed limits method with a strong comparison result.

However, in the present discontinuous framework, although classical viscosity solutions---\CVS in
short---still have good stability properties as described in Section~\ref{sect:stab}, the lack of
uniqueness makes this stability far less effective: the two half-relaxed limits are lying between
the minimal one $\Um$ and the maximal one $\Up$ and one cannot really obtain the convergence in that
way, except if $\Up=\Um$.

An interesting idea is to turn to flux-limited solutions for which a general comparison result
holds. But, in order to identify the limit of the vanishing viscosity method, a limit flux~limiter
is required and to the best of our knowledge, there is no obvious way to determine it. Actually we
refer the interested reader to Section~\ref{TMGD} for a discussion on more general discontinuities
where the problem is still open.

We also refer anyway to \cite{IM,IM-md} for general stability results for \FLS and to Camilli, Marchi
and Schieborn \cite{CMS-VV} for the first results on the convergence of the vanishing viscosity
method.

In this book, we give several different proofs of the vanishing viscosity result. Tthe first one
below is inspired from \cite{BBCI} and uses only the properties of $\Up$ as flux-limited solution. 

Contrary to the proof relying on \JVS via the Lions-Souganidis approach, the arguments we use in
this section strongly rely on the structure of the Hamiltonians and on the arguments of the
comparison proof. It has the advantage anyway to identify the limit in terms of control problems. An
other way to do the proof goes through the connections between the Kirchhoff condition and
Flux-Limited Conditions (See Section~\ref{sec:KCvsFL}).

\begin{theorem}\emph{--- Vanishing viscosity limit via flux-limited solutions.}\label{teo:viscous}\smsp
    Assume that \GACC holds. For any $\eps>0$, let $u^\eps$ be a viscosity solution of
    \begin{equation}\label{pb:viscous1}
       u^\eps_t -\eps \Delta u^\eps + H (x,t,u^\eps,Du^\eps) = 0\quad\text{in}
        \quad\R^N \times (0,\Tf)\;,
    \end{equation}
    \begin{equation}\label{pb:viscousid1}
       u^\eps(x,0) = u_ 0(x) \quad\text{in}\quad\R^N \;,
    \end{equation}
    where $H=H_1$ in $\Omega_1$ and $H_2$ in $\Omega_2$, and $u_0$ is bounded continuous function in
    $\R^N$. If the $u^\eps$ are uniformly bounded in $\R^N \times (0,\Tf)$ and $C^1$ in $x_N$ in a
    neighborhood of $\H$, then, as $\eps\to0$, the sequence $(u^\eps)_\eps$ converges locally
    uniformly in $\R^N \times (0,\Tf)$ to $\VFp$, the maximal Ishii subsolution of
    \eqref{pb:half-space}.
\end{theorem}
\begin{remark}\label{rem:reg}
    A priori \eqref{pb:viscous1}-\eqref{pb:viscousid1} is a uniformly parabolic problem and the
    regularity we assume on $(u^\eps)$ is reasonable. Indeed the function $u^\eps$ is expected to
    be $C^1$ since it is also expected to be in $W^{2,r}_{\rm loc}$ (for any $r>1$). On the other
    hand, it is worth pointing out that, as long as $\eps>0$, it is not necessary to impose a
    condition on $\H$ because of the strong diffusion term: a codimension 1 set is not ``seen''
    by the diffusive equation.
\end{remark}

\begin{proof}
    We first recall that, by Theorem~\ref{thm:Vp-max}, $\Up$ is the maximal subsolution---and Ishii
    solution---of \eqref{pb:half-space} and we proved in Theorem~\ref{compBBC-IM} that it is the
    unique flux-limited solution of \HJgen-\FL with $G=\HTreg$. We recall that the flux-limited
    condition consists in complementing \HJgen with the condition
    \[
    \max \Big( u_t+\HTreg(x,t,D_\H u) , u_t+ H_1^+(x, t,D_x u) , u_t+
    H_2^-(x, t, D_x u) \Big) = 0 \: \ \mbox{ on } \H \times (0,\Tf)\; ,
    \]
    in the sense of Definition~\ref{defiFL}. We refer to Section~\ref{sect:stab} for a definition of
    the usual half-relaxed limits
    \[
    \uu(x,t):=\limiinf u^\eps(x,t)\;,   \quad 
    \ou(x,t):=\limssup  u^\eps(x,t) \: . 
    \]

    \noindent\textbf{(a)} \emph{Reduction of the proof  ---} 
    We observe that we only need to prove the following inequality
    \begin{equation} \label{equ:tesi} 
     \Up(x,t) \leq \underline{u}(x,t) \quad \mbox{ in } \R^N \times [0,\Tf) .
    \end{equation}
    Indeed, the maximality of $\Up$ implies $\overline{u}(x,t)\leq \Up(x,t)$ in $\R^N \times
    [0,\Tf)$. Moreover, by definition we have $\overline{u}(x,t)\geq \underline{u}(x,t)$ in $\R^N
    \times (0,\Tf)$, therefore if we prove \eqref{equ:tesi} we can conclude that $\Up(x,t) \leq
    \underline{u}(x,t) \leq \overline{u}(x,t)\leq \Up(x,t)$ which implies that $(u^\eps)_\eps$
    converges locally uniformly to $\Up$ in $\R^N \times [0,\Tf)$.

    In order to prove the inequality, $\Up \leq \underline{u}$ in $\R^N \times [0,\Tf)$, we are
    going to make several reductions along the lines of Chapter~\ref{chap:pde.tools} by changing
    $\Up$ but we keep the notation $\Up$ for the changed function for the sake of simplicity of
    notations. In the same way, we should argue on the interval $[0,T']$ for $0<T'<\Tf$ but we keep
    the notation $\Tf$ for $T'$.

    First, thanks to the localization arguments of Chapter~\ref{chap:pde.tools}, we can assume that
    $\Up$ is a strict subsolution such that $\Up (x,t) \to -\infty$ as $|x|\to +\infty$, uniformly
    \wrt $t\in [0,\Tf]$. Therefore there exists $(\bx,\bt) \in \R^N \times [0,\Tf]$ such that
    $$
     M:=\Up (\bx,\bt)- \underline{u}(\bx,\bt)=\sup_{(x,t) \in \R^N \times [0,\Tf]} \:  
     \big( \Up (x,t) -  \underline{u}(x,t) \big)  \:.
    $$
    We assume by contradiction that $M >0$ and of course this means that $\bt>0$. The cases when
    $\bx\in \Omega_1$ or $\bx \in \Omega_2$ can be treated by classical methods, hence we may assume
    that $\bx \in \H$.

    Next, by the regularization arguments of Chapter~\ref{chap:pde.tools} we can assume in addition
    that $\Up$ is $C^1$-smooth at least in the $t,x_1,\dots,x_{N-1}$ variables. Finally we can
    suppose that $(\bx,\bt)$ is a  strict maximum point of $\Up-\underline{u}$.

    \bigskip

    \noindent\textbf{(b)} \emph{Construction of the test-function ---} 
    Since $\Up$ is $C^1$ in the $(t,x')$-variables, the strict flux-limited subsolution condition
    can be written as  
    $$
    (\Up)_t(\bx,\bt)+\HTreg(\bx,\bt,D_{x'}\Up(\bx,\bt)) \leq -\eta \:, 
    $$
    where $\eta>0$ measures the strict subsolution property. Therefore 
    $$
    \HTreg(\bx,\bt,D_{x'}\Up(\bx,\bt)) \leq -(\Up)_t(\bx,\bt)-\eta  \: ,
    $$
    and, as in the proof of Theorem~\ref{comp-IM}, there exist two
    solutions $\lambda_1,\lambda_2$, with $\lambda_2 < \lambda_1$, of the equation
    $$
    \Htireg\Big(\bx,\bt,D_{x'}\Up(\bx,\bt)+\lambda e_N\Big) = -(\Up)_t(\bx,\bt)-\eta/2 \;.
    $$
    Notice that, since $\bx, \bt$, $a=-(\Up)_t(\bx,\bt)$ and $p'=D_{x'}\Up(\bx)$ are
    fixed, $\lambda_1, \lambda_2$ are independent of the parameter $\e>0$ that is to come
    below. 

    We proceed now with the construction of the test-function: let $\chi(x_N,y_N)$ be defined as in
    \eqref{defchila} and 
    $$ \psi_\eps (x,y,t,s):=\frac{|t-s|^2}{\eps^{1/2}}+\frac{|x'-y'|^2}{\eps^{1/2}}+
    \chi(x,y)+\frac{|x_N-y_N|^2}{\eps^{1/2}}\;.$$
    Note that  $\psi_\eps(\cdot,y,\cdot,s),  \psi_\eps(x,\cdot,t,\cdot)\in \PC1$.

    Since $(\bx,\bt)$ is a strict global maximum point of $\Up-\underline{u}$ while
    $\underline{u}(\bx,\bt)=\limiinf u^\eps (\bx,\bt)$, the function
    $\Up(x,t)-u^\e(y,s)-\psi_\eps (x,y,t,s)$ has local maximum points $(x_\eps ,y_\eps
    ,t_\eps ,s_\eps )$ which converge to $(\bx,\bx,\bt,\bt)$. For the sake of simplicity of
    notations, we drop the $\eps$ and just denote by $(x,y,t,s)$ such a maximum point.

    \bigskip

    \noindent\textbf{(c)} \emph{Getting a contradiction ---} 
    We now consider 3 different cases, depending on the position of $(x,y,t,s)$.

    \smallskip

    \noindent\textbf{Case 1}: $x_N>0$ and $y_N\leq0$ (or $x_N<0$ and $y_N\geq0$).\\[2mm]
    We use the subsolution condition for $\Up$ in $\Omega_1$: recalling that $\Up$ is $C^1$-regular
    in the $(t,x')$-variables, we write the condition as
    $$
    (\Up)_t(x,t)+H_1 \left(x,t,D_{x'}\Up(x,t) + \lambda_1 e_N + 
    \frac{2(x_N-y_N)}{\eps^{1/2}}\right) \leq -\eta  \; , 
    $$
    where we have used the regularity of $\Up$ to deduce that
    \begin{equation}\label{prop-reg-Up}
        (\Up)_t(x,t) = \frac{2(t-s)}{\eps^{1/2}}\quad\hbox{and}\quad
        D_{x'}\Up(x,t) = \frac{2(x'-y')}{\eps^{1/2}}\; .
    \end{equation}
    Moreover, using further the regularity of $\Up$ and recalling that $(\Up)_t$ and $D_{x'}\Up$ are
    continuous not only in $t,x'$ but also $x_N$, we have $(\Up)_t(x,t)=(\Up)_t(\bx,\bt)+o_\e(1)$,
    $D_{x'}\Up(x,t)=D_{x'}\Up(\bx,\bt)+o_\e(1)$. Therefore,
    $$
    (\Up)_t(\bx,\bt) +H_1 \left(x,t,D_{x'}\Up(\bx,\bt) + 
    \lambda_1 e_N + \frac{2(x_N-y_N)}{\eps^{1/2}}\right) \leq -\eta +o_\e(1)  \;.
    $$
    Next, using that $H^-_1$ is non decreasing in $p_N$, $H_1^-\leq H_1$ and $(x_N-y_N)>0$ we get
    from the above property
    $$\begin{aligned}
    H_1^- \left(x,t,D_{x'}\Up(\bx,\bt) + \lambda_1 e_N\right)  & \leq 
    H_1^- \left(x,t,D_{x'}\Up(\bx,\bt) + \lambda_1 e_N + \frac{2(x_N-y_N)}{\eps^{1/2}}\right) \\
    & \leq  -(\Up)_t(\bx,\bt)-\eta +o_\e(1)\;.
    \end{aligned}$$
    From this inequality, since $D_{x'}\Up(\bx,\bt) + \lambda_1 e_N$ remains bounded with respect to
    $\e$, using the continuity of $H_1^-$ yields 
    $$ H_1^- \left(\bx,\bt,D_{x'}\Up(\bx,\bt) + \lambda_1 e_N \right)
    \leq  -(\Up)_t(\bx,\bt)-\eta +o_\e(1)\;.$$
    The contradiction is obtained for $\e$ small enough from the fact that, by construction of
    $\lambda_1$,
    $$H_1^- \left(\bx,\bt,D_{x'}\Up(\bx,\bt) + \lambda_1 e_N\right) 
    = -(\Up)_t(\bx,\bt)-\eta/2\;.
    $$
    The case $x_N<0$ and $y_N\geq0$ is completely similar, using $H_2$ instead of $H_1$.  

    \medskip

    \noindent\textbf{Case 2}: $x_N=0$ and $y_N>0$ (or $<0$).\\[2mm]
    We use the supersolution viscosity inequality for $u^\eps$ at $(y,t)$, using
    \eqref{prop-reg-Up}:
    \begin{equation}\label{ineq.ud}
      O(\eps^{1/2})+(\Up)_t(x,t) +  H_1\Big(y,s,D_{x'}\Up(x,t)+\lambda_1e_N+
      \frac{2(x_N-y_N)}{\eps^{1/2}}+o_\e(1)\Big)\geq0\;.
    \end{equation}
    Notice that, using the arguments of Case 1 and the fact that here $x_N-y_N=-y_N<0$, 
    we are led by the definition of $\lambda_1$ to
    $$
    O(\eps^{1/2})+(\Up)_t(x,t) + H_1^-\Big(y,s,D_{x'}\Up(x,t)+\lambda_1e_N+
    \frac{2(x_N-y_N)}{\eps^{1/2}}+o_\e(1)\Big)<0\;,
    $$
    from which we deduce that \eqref{ineq.ud} holds true with $H^+_1$. 

    Moreover, by the subsolution condition  of $\Up$ on $\H$ we have 
    $$
    (\Up)_t(x,t)  + H^+_1\Big(x,t,D_{x'}\Up(x,t)+\lambda_1e_N+
    \frac{2(x_N-y_N)}{\eps^{1/2}}+o_\e(1)\Big)\leq-\eta\;,
    $$
    therefore the conclusion follows by standard arguments  putting together the two inequalities
    for $ H^+_1$ and letting $\eps$ tend to zero. If $y_N< 0$, we can repeat the same argument using
    this time $H^-_2$ instead of $H_1^+$.

    \medskip

    \noindent \textbf{Case 3}: $x_N=y_N=0$.\\[2mm]
    Let us remark that this case is not possible. Indeed the  maximum point property on
    $\Up-u^\eps-\psi_\eps$ implies that $0$ is a minimum point of $z_N \mapsto
    u^\eps((y',z_N),s)+\psi_\eps(x,(y',z_N),t,s) )$. But, by definition of $\psi_\e$ and in
    particular of $\chi$, this also means that we have a minimum point for the function
    $$
        \zeta:z_N\mapsto u^\eps((y',z_N),s)-h(z_N)+\frac{|z_N|^2}{\eps^{1/2}}\; .
    $$
    Both $z_N\mapsto |z_N|^2$ and $u^\e$ are $C^1$-smooth, but the function $h$ is only Lipschitz
    continuous at $z_N=0$. So, using that the left derivative of $\zeta$ is
    negative while the right one is positive leads to $-h'(0^-)\leq -h'(0^+)$, \ie
    $\lambda_2\geq \lambda_1$. But this contradicts the construction of function $\chi$ which
    requires $\lambda_2 < \lambda_1$.
\end{proof}

\section{Classical viscosity solutions as flux-limited solutions}

The aim of this section is to show that, under suitable assumptions, a classical viscosity sub or
supersolution of 
\begin{equation}\label{eqn:cl-E}
u_t + H(x,t,u,D_x u) = 0 \quad \hbox{in  }\R^N \times (0,\Tf)\;,
\end{equation}
where $H$ is a continuous quasi-convex Hamiltonian, is a \FLSub or \FLSup of the problem with $H_1=H_2=H$ and $G=\HT$ where, for $x\in \H$, $t\in [0,\Tf]$,
$r\in \R$ and $p'\in \H$
$$ \HT(x,t,r,p')= \min_{s\in\R} H (x,t,r,p'+s e_N)\; .$$
We refer the reader to Section~\ref{upoH} and in particular to Lemma~\ref{lem:H1m.H2p.a} for a motivation
of the definition of $\HT$ in the convex case but we are going to consider below the more general quasi-convex case.

The precise result is the
\begin{proposition}\label{prop:CSasFLS}\emph{--- Classical Ishii solutions and flux-limited
    solutions.}\smsp
    Assume that \GAQC holds with $H_1=H_2=H$ and that $G=\HT$ satisfies \GAGFL.
Then $u$ is a classical Ishii subsolution \resp{supersolution} of \eqref{eqn:cl-E} if and only if it is a \FLSub \resp{\FLSup}
of \HJgen-\FL with $H_1=H_2=H$ and $G=\HT$.
\end{proposition}

The interest of this result is to be able to introduce an artificial discontinuity when it is useful. We refer the reader to
Section~\ref{sec:using-equiv} for an example of such situation.

\begin{proof}To prove that a \FLSub (or \FLSup) is a classical Ishii subsolution (or supersolution) is easy using that (i) $C^1$ 
test-functions are PC$^1$ test-functions and (ii) $\max(H^+,H^-, \HT)=H$. 

We only prove the converse for the subsolution case, the supersolution one being essentially analogous; we just provide below a tiny
additional argument to treat this supersolution case. Of course, only the properties
on $\H\times (0,\Tf)$ are different and therefore we concentrate on this case.

Let $u$ be a classical Ishii subsolution of \eqref{eqn:cl-E} and let $(x,t)\in \H\times (0,\Tf)$ be a strict local maximum point of
$u-\varphi$ where $\varphi=(\varphi_1,\varphi_2) \in \PC1$. We have to look at two different cases
\begin{enumerate}
\item[(i)] $\lambda:=\dfrac{\partial \varphi_1}{\partial x_N}(x,t)\leq \mu:=\dfrac{\partial \varphi_2}{\partial x_N}(x,t)$.
\item[(ii)] $\lambda>\mu$.
\end{enumerate}

Case~$(i)$ is easy: if $p'=D_{x'} \varphi (x,t)$ and $p_t=\varphi_t (x,t)$ then, for any $\lambda \leq \tau \leq \mu$, $((p',\tau),p_t) \in D_{\Omegb_i\times (0,\Tf)}^+ u (\xb,\tb)$ for $i=1$ and $i=2$; hence $((p',\tau),p_t) \in D_{\R^N \times (0,\Tf)}^+ u (\xb,\tb)$ and therefore
$$ p_t + H(x,t,u(x,t), p' +\tau e_N)\leq 0\; .$$
Using that $\max(H^+,H^-, \HT)=H$, we easily obtain the desired inequalities by choosing $\tau =\lambda$ and then $\tau =\mu$.

Case~$(ii)$ is more tricky: by Lemma~\ref{diff-twod}, we can assume without loss of generality that $ \varphi= \chi+\psi$
where $\psi$ is $C^1$ in $\R^N \times (0,\Tf)$ and
$$ \chi(x_N):= \begin{cases}
\lambda x_N & \hbox{if  $x_N\geq 0$}\; ,\\
\mu x_N & \hbox{if  $x_N\leq 0$\; .}
\end{cases} 
$$
We mollify the function $\chi$ by using a mollifying kernel with compact support and we obtain a sequence of
$C^1$-functions $(\chi_\e)_\e$ and then a sequence $(\varphi_\e)_\e$ given by $\varphi_\e=\chi_\e +\psi$.
Moreover, by standard convolution arguments, we have
$$ \mu  \leq \frac{\partial \chi_\e}{\partial x_N}(x_N) \leq \lambda \quad \hbox{for any }x_N\; .$$
Let $(\xe,\te)$ be a sequence of maximum points of $u-\varphi_\e$ which converges to $(x,t)$ and such that $u(\xe,\te) \to u(x,t)$
(such sequence exists since $(x,t)$ is a strict local maximum point of $u-\varphi$ and $\varphi_\e\to \varphi$ locally uniformly). We have 
$$ \frac{\partial \psi}{\partial t}(\xe,\te) + H(\xe,\te, u(\xe,\te), D_{x'}\psi (\xe,\te)+ \frac{\partial \chi_\e}{\partial x_N}(\xe,\te)e_N)\leq 0\; .$$
Introducing
$$ \tilde H(\tau):= \frac{\partial \psi}{\partial t}(x,t) + H(x,t, u(x,t), D_{x'}\psi (x,t)+ \tau e_N)\; ,$$
and denoting respectively by $\tilde H^+$, $\tilde H^-$, $ \tilde H_T$, functions which are defined in the same way, replacing $H$ by
$H^+$, $H^-$ or $\HT$, we deduce from the continuity of $H$, the above properties and the $C^1$ character of $\psi$, that
$$ \tilde H(\frac{\partial \chi_\e}{\partial x_N}(\xe,\te)) \leq o_\e(1)\; .$$
This inequality can be rewritten as
$$ \max(\tilde H^+,\tilde H^-,\tilde H_T )(\frac{\partial \chi_\e}{\partial x_N}(\xe,\te))\leq o_\e(1)\; ,$$
and using the monotonicity of $\tilde H^+,\tilde H^-$, we have, because $\HT$ is independent of the $x_N$-derivative
$$ \max(\tilde H^+ (\lambda),\ \tilde H^-(\mu),\ \tilde H_T )\leq o_\e(1)\; ,$$
and we conclude by letting $\e \to 0$.

For supersolutions, the analogue of Case~$(ii)$ is treated exactly in the same way. Case~$(i)$--which is now $\lambda\geq \mu$--required the following additional arguments: with the above notations, we have
$$ \max(\tilde H^+ (\tau),\ \tilde H^-(\tau),\ \tilde H_T )\geq 0\; ,$$
for any $\mu \leq \tau \leq \lambda$ and we have three cases
\begin{enumerate}
\item If $\tilde H_T \geq 0$, we are done.
\item If $\tilde H_T < 0$, by choosing $\tau = \lambda$, we have $\max(\tilde H^+ (\lambda),\ \tilde H^-(\lambda))\geq 0$. If 
$\tilde H^+ (\lambda)\geq 0$, we are done. In the same way, by choosing $\tau = \mu$, we have $\max(\tilde H^+ (\mu),\ \tilde H^-(\mu))\geq 0$. If  $\tilde H^- (\mu)\geq 0$, we are done.
\item If $\tilde H_T < 0$, $\tilde H^+ (\lambda)< 0$ and $\tilde H^- (\mu)< 0$, then necessarily $\tilde H^-(\lambda)\geq 0$ and $\tilde H^+ (\mu)\geq 0$. Hence
$$ (\tilde H^+ -\tilde H^-) (\lambda)< 0\quad , \quad (\tilde H^+ -\tilde H^-) (\mu)> 0\; ,$$
and there exists $\tau \in (\mu,\lambda)$ such that $\tilde H^+ (\tau) =\tilde H^- (\tau)$. But, for such $\tau$, we have
$\tilde H^+ (\tau) =\tilde H^- (\tau)=\HT$. Therefore using such $\tau$ in the above inequality yields $\HT\geq 0$, a contradiction
which means that we are in one of the two first cases.
\end{enumerate}
And the proof is complete.
\end{proof}

\section{Extension to second-order equations (I)}

In this section, we consider second-order equations of the form 
$$
    u_t+H_i(x,t,Du)-{\rm Tr}(a_i(x)D^2u) = 0 \quad\hbox{in  }\Omega_i\times (0,\Tf)\; , 
$$
where $a_i$ $(i=1,2)$ are continuous functions which are assumed to be on the standard form, \ie
$a_i=\sigma_i\cdot\sigma_i^T$ where $\sigma_i^T$ is the transpose matrix of $\sigma_i$. We suppose
that the $\sigma_i$'s are bounded, Lipschitz continuous functions and in order that the definition
of flux-limited solutions make sense, the following property has to be imposed 
$$
\sigma_i((x',0))=0\quad \hbox{for $i=1,2$ and for all $x'\in \R^{N-1}$.}
$$

The main question we address here concerns the comparison result in this framework.
There are several difficulties that we list below:
\begin{enumerate}
\item[$(i)$] in general, we cannot regularize the subsolution as we did above;
\item[$(ii)$] because of the second-order term, the normal controllability cannot be used
    efficiently outside $\H$;
\item[$(iii)$] a two-parameter proof as in the non-convex case is difficult to handle with the
    second-order term.
\end{enumerate}

We take this opportunity to remark that the above comparison proofs has several common points with the
comparison proof for nonlinear Neumann boundary conditions: in fact, it can be described as a
``double Neumann'' proof since $H_1^-$ (almost) plays the role of a Neumann boundary condition for
the equation in $\Omega_2$ while conversely $H_2^+$ (almost) plays the role of a Neumann boundary
condition for the equation in $\Omega_1$, see Proposition~\ref{prop:dyn.neumann} for more
explanations.

There is anyway a crucial additional difficulty:  $H_1^-,
H_2^+$ are NOT strictly monotone functions \wrt the normal gradient direction. Therefore, if a
general ``one-parameter proof'', avoiding the use of $\gamma \ll \e$ may be possible, it is
probably rather technical and may require additional assumptions on Hamiltonians $H_i$.

Instead, the following result gives some conditions under which the proof of
Theorem~\ref{comp-IM-nc} still works.
\begin{theorem}\emph{--- Comparison principle in the second-order case.}\smsp \label{comp-IM-so} 
    Under the assumptions of Theorem~\ref{comp-IM-nc}, the result of Theorem~\ref{comp-IM} is valid
    provided that the two following assumptions hold, for $i=1,2$, in a neighborhood of $\H$:\\[2mm] 
    $(i)$ $H_i(x,t,p)=H_{i,1}(x',t,p')+H_{i,2}(x_N,p_N)$,\\
    $(ii)$ $\sigma_i=\sigma_i(x_N)$ with $\sigma_i(0)=0$, 
    $\sigma_i$ being locally Lipschitz continuous and bounded.
\end{theorem}

It is worth pointing out that this result holds for non-convex Hamiltonians, but requires rather
restrictive assumptions on $H_i$ and $\sigma_i$. We refer to Imbert and Nguyen \cite{IN} for
general results for second-order equations in the case of {\em networks} where not only comparison
results are obtained but the notions of \FLS and \JVS are discussed and applications are given.

\begin{proof}
    The proof follows readily the proof of Theorem~\ref{comp-IM-nc}, we just add here some
    comments:
    \begin{enumerate}
        \item[--] The structure conditions we impose on $(H_i,\sigma_i)_{i=1,2}$ ensures that we
            can perform a regularization of the subsolution by sup-convolution in the spirit of
            Proposition~\ref{reg-by-ic}: in particular, the Hamiltonians both satisfy
            \TCs. This is the first reason to assume $(i)$ and $(ii)$.
            \index{Regularization of subsolutions}
        \item[--] Once this regularization is done, we still have to control the dependence in the
            derivatives in $x_N$ (or all the terms involving the parameter $\gamma$): this is where
            the special dependence in $x_N$ of $H_i$ and $\sigma_i$ plays a role.
        \item[--] In all the steps where the properties of $\lambda_1,\lambda_2$ are crucial, the
            second-order term is small since $|\sigma_i(x_N)| = O(|x_N|)$ and therefore $|a_i(x_N)|
            = O(x_N^2)$. This can be combined with the facts that 
            $$\frac{|(\xe)_N-(\ye)_N|^2}{\gamma^2}\to 0\quad \hbox{as  }\gamma \to 0\; ,$$
            and the second-order derivatives are a $O(\gamma^{-2})$.
    \end{enumerate}\
\end{proof}

\begin{remark} 
    Anticipating the main result of Section~\ref{sec:KCvsFL} showing that the Kirchhoff boundary
    conditions is equivalent to a flux-limited boundary condition with $G=\HTreg$ under the
    assumptions of Theorems~\ref{comp-IM} or~\ref{comp-IM-nc}, these two results also provide the
    comparison for the (KC)-condition. The proof(s) would apply readily if we were able to show that
    we can choose $\lambda_1 > \lambda_2$ in the test-function (the function $\chi$) but this is not
    obvious at this point and this property will be clarified in Section~\ref{sec:KCvsFL}.
\end{remark}

\chapter{Junction Viscosity Solutions}
\label{chap:JVS}

\abstract{This chapter is devoted to study junction viscosity solutions \`a la Lions-Souganidis for
continuous Hamiltonians: definition, stability and comparison properties are described in details.}

\index{Kirchhoff condition}
Even if flux-limited viscosity solutions have their advantages, it may seem more natural to consider
a definition of viscosity solution with a $\min/\max$ condition on the junction involving $H_1$ and
$H_2$ instead of their nondecreasing/nonincreasing parts.
 
In the next sections, we present the general notion of junction viscosity solutions, which is called
``relaxed solution'' in \cite{IM}. However, because of the similarity to the classical notion of
viscosity solutions, it seems to us that ``junction viscosity solutions'' is more appropriate.  

\section{Definition and first properties}
\label{sec:def-JVS}

\index{Junction viscosity solutions!definition}
We introduce the notion of junction viscosity sub/supersolution for \HJgen associated with a \GJC
given by a nonlinearity $G$ as follows
\footnote{We recall that we assume that $G(x,t,r,a,p',b,c)$ is independent of $r$.}
\begin{definition}\emph{--- Junction Viscosity Solutions.}\\
    \label{defiJVS}
    A locally bounded function $u: \R^N \times (0,\Tf) \rightarrow \R$ is a \JVSub of
    \HJgen-\GJC if it is a classical viscosity subsolution of \HJgen and if, for any test-function
    $\psi=( \psi_1, \psi_2) \in \PC1$ and any local maximum point $(x,t) \in \H\times (0,\Tf) $ of
    $u^*-\psi$ in $\R^N\times (0,\Tf)$,
    \begin{equation} \label{cvsub-n}
        \min \Big(G(x,t,\psi_t,D_\H \psi,\frac{\partial \psi_1}{\partial n_1},
        \frac{\partial \psi_2}{\partial n_2}),  \psi_t+ H_1(x,t, u^*, D\psi_1) ,  
        \psi_t+ H_2(x,t, u^*, D\psi_2) \Big) \leq 0  \: ,
    \end{equation}
    where $u^*$ and the derivatives of $\psi,\psi_1,\psi_2$ are taken at $(x,t)$.\\[2mm]
    A locally bounded function $v: \R^N \times (0,\Tf) \rightarrow \R$ is a \JVSup of
    \HJgen-\GJC if it is a classical viscosity supersolution of \HJgen and if, for any test-function
    $\psi=( \psi_1, \psi_2) \in \PC1$ and any local minimum point $(x,t) \in \H\times (0,\Tf) $ of
    $v_*-\psi$ in $\R^N\times (0,\Tf)$,
    \begin{equation} \label{cvsuper-n}
        \max \Big(G(x,t,\psi_t,D_\H \psi,\frac{\partial \psi_1}{\partial n_1},
        \frac{\partial \psi_2}{\partial n_2}),  \psi_t+ H_1(x,t, v_*, D\psi_1) ,  
        \psi_t+ H_2(x,t, v_*, D\psi_2) \Big) \geq 0  \: ,
    \end{equation}
    where $v_*$ and the derivatives of $\psi,\psi_1,\psi_2$ are taken at $(x,t)$.\\
    A \JVS (i.e. a junction viscosity solution) is a locally bounded function which is both \JVSub and
\JVSup.
\end{definition}

As in the case of \FLSub and \FLSup, we can define \JVSub and \JVSup using the notions
of sub and superdifferentials. With the notations of Proposition~ \ref{defiFL2}, we have the
\footnote{Again we formulate the result for \usc subsolution and \lsc supersolution but the reader
can easily transpose it to general sub and supersolutions}
\begin{proposition}\label{defiJVS-diff}\emph{--- Junction viscosity solutions via sub
    superdifferentials.}\smsp
\noindent An \usc, locally bounded function $u: \R^N \times (0,\Tf) \rightarrow \R$ is a
            \JVSub of  \HJgen-\GJC if and only if 
            \begin{enumerate}
                \item[$(i)$] for any $(x,t) \in Q_i$ ($i=1,2$) and any $(p_x,p_t)\in D_{\overline{Q_i}^\ell}^+ u (x,t)$
            $$
            p_t + H_i (x,t,u(x,t),p_x) \leq 0 \; ,$$
        \item[$(ii)$] for any $(x,t) \in \H\times (0,\Tf)$ and for any $p_\H \in \H$, $p_1,p_2,p_t
            \in \R$ such that $((p_\H,p_i),p_t)\in D_{\overline{Q_i}^\ell}^+ u (x,t)$ for $i=1,2$,
         $$
             \min_i \Big(G(x,t,p_t,p_\H ,p_1,p_2),  
              p_t+ H_i(x,t, u(x,t), p_\H+p_i e_N) \Big) \leq 0\;.
            $$
            \end{enumerate}

            \noindent  A \lsc, locally bounded function $v: \R^N \times (0,\Tf) \rightarrow \R$
                is a \JVSup of  \HJgen-\GJC if and only if, for $(x,t) \in \R^N\times (0,\Tf) $,
                \begin{enumerate}
                    \item[$(i)$] for any $(x,t) \in Q_i$ ($i=1,2$) and for any $(p_x,p_t)\in
            D_{\overline{Q_i}^\ell}^-v (x,t)$
            $$
            p_t + H_i (x,t,v(x,t),p_x) \geq 0 \; ,$$
        \item[$(ii)$] for any $(x,t) \in \H\times (0,\Tf)$ and for any $p_\H \in \H$, $p_1,p_2,p_t
            \in \R$ such that $((p_\H,p_i),p_t)\in D_{\overline{Q_i}^\ell}^-v (x,t)$ for $i=1,2$,
                               $$
             \max_i \Big(G(x,t,p_t,p_\H ,p_1,p_2),
              p_t+ H_i(x,t, v(x,t), p_\H+p_i e_N) \Big) \geq 0 \;.
            $$
                \end{enumerate}
\end{proposition}

As for Proposition~\ref{defiFL2}, we leave the proof of this result to the reader since it is an
easy consequence of Lemma~\ref{subdiff-hp-Omega} and Lemma~\ref{diff-twod}. We again point out that
this equivalent definition via sub and superdifferentials allows to show that instead of using
general PC$^1$ test-functions, we may only use test-functions of the form $\chi(x_N)+\varphi (x,t)$
where $\chi \in {\rm PC}^1(\R)$ and $\varphi \in C^1(\R^N \times (0,\Tf))$. The reader will notice
that we mainly use test-function of this form in the comparison result but this property
is also useful to simplify the proofs of several results.

Before considering the regularity properties of \JVSub and \JVSup, we point out that one of the
advantages of the notion of junction viscosity solution is that it can be applied to a wider class of
junction conditions without any convexity/quasi-convexity type assumption.  On the other hand, its
similarity with the classical notion of viscosity solutions should easily convince the reader that
the notion enjoys the stability properties of classical viscosity solutions.

\subsection{Lack of regularity of subsolutions}

This notion has a slight defect since \usc junction viscosity subsolutions are not necessarily regular,
contrarily to flux-limited solutions, because of the ``$\min$'' in the definition. To show it, we
consider the following $1$-d example
$$\begin{cases}
    u_t +|u_x|=0 & \hbox{in  }\R\setminus\{0\}\times (0,+\infty)\; ,\\
    u_t(0,t)=0  & \hbox{in  } (0,+\infty)\; ,\\
    u(x,0)=-|x| & \hbox{in  }\R\; .
\end{cases}
$$
It is worth pointing out that this problem is far from being pathological since $H_1(p)=H_2(p)=|p|$
satisfy all the ``good assumptions'', in particular \NCe. One checks easily that the expected
solution is $U(x,t)=-|x|-t$ but we also have the non-regular subsolution given by 
$$u(x,t)=\begin{cases}
    U(x,t) & \hbox{if  } x\neq 0\; ,\\
    0 & \hbox{for  }x=0\; .
\end{cases}$$
It is clear that $u$ is \usc and a subsolution for $x\neq 0$, and it is a subsolution for $x=0$
because $u_t(0,t)\equiv 0$ and the ``min'' in the definition allows such inexpected feature.

\subsection{The case of Kirchhoff-type conditions}

We refer to Section~\ref{sect:types.junctions} for the complete and precise definitions of different
junction conditions on the inferface. Let us just recall here that Kirchhoff-type conditions 
essentially satisfy (dropping the $u$-dependance)
\begin{equation}
        G(x,t,a_1,p',b_1,c_1)\,-\,  G(x,t, a_2,p',b_2,c_2)\geq \alpha(a_1-a_2) + 
    \beta(b_1-b_2)+ \beta(c_1-c_2)
\end{equation}
for some $\alpha\geq0,\ \beta>0$. Of course the typical example is the standard Kirchhoff condition
for which $G(x,t,a,p',b,c)=b+c$, encoding $\partial u/\partial n_1+\partial u/\partial n_2=0$.
\index{Junction viscosity solutions!regularity of subsolutions}

\begin{proposition}\label{prop:JVreg}\emph{--- Regularity of subsolutions.}\smsp
    Assume that $H_1$, $H_2$, $G$ are continuous functions and that $H_1$, $H_2$ satisfy \NCHJ. Then
    junction viscosity subsolutions are regular provided \GJC is of Kirchhoff type.  
\end{proposition}

\begin{proof}
    We only provide the proof in the subsolution case, the supersolution one being analogous. Assume
    that $u: \R^N \times (0,\Tf) \rightarrow \R$ is an \usc \JVSub of \HJgen-\GJC and let $(x,t)$ be
    a point of $\H \times (0,\Tf)$. We argue by contradiction: if, for instance, $u$ is not
    $\Omega_1$-regular at $(x,t)$, this means that
    \begin{equation}\label{jump-sub}
    u(x,t) > \limsup_{\substack{ (y,s)\to (x,t)\\ y\in \Omega_1}}u(y,s)\, .
    \end{equation}
    We introduce the function
    $$ \Psi:(y,s) \mapsto u(y,s) -\frac{|y-x|^2}{\eps^2}-\frac{|s-t|^2}{\eps^2}-C_1 (y_N)_+ -C_2 (y_N)_-\;,$$
    where $0<\eps \ll 1$ and $C_1\in \R$, $C_2>0$ are constants to be chosen. 
    Notice that $y\mapsto C_1 (y_N)_+ +C_2 (y_N)_-$ belongs to $\PC1$.

    Choosing $\eps$ small enough and $C_1=0$, $\Psi$ has a maximum point $(\xe,\te)$ near $(x,t)$ and
    $(\xe,\te)\to (x,t)$, $u(\xe,\te)\to u(x,t)$ as $\eps\to 0$. We see that if $C_2$ is large enough,
    the $H_2$-subsolution inequality cannot hold, therefore $(\xe,\te)\in \H$.  Moreover, if
    $\eps$ is small enough, \eqref{jump-sub} is also true at $(\xe,\te)$ and, as a consequence,
    $(\xe,\te)$ is a local maximum point of $\Psi$ for any $C_1\in \R$. 

    In the same way, choosing now $C_1<0$ large enough implies that the $H_1$-subsolution inequality
    cannot hold. But neither can the $G$-one, provided $G$ is of Kirchhoff-type: since $G$ behaves
    like $K_\e+\beta(C_2-C_1)$ for some constant $K_\e$, it is strictly positive if $C_1$ is very
    negative. 

    Hence, none of the subsolution inequalities can hold on the interface and we get the desired
    contradiction.  
\end{proof}

On the contrary,  junction viscosity supersolutions are not necessarily regular, even if \GJC is of
Kirchhoff type as shown by the following example.

\begin{example} The solution $u:\R \times [0,+\infty[$ of
$$ u_t +|u_x|=0 \quad \hbox{in  }\R \times (0,+\infty)\; ,$$
$$ u(x,0)=u_0 \quad \hbox{in  }\R \; ,$$
where $u_0(x):= (1-|x|)_+$ is given by $u(x,t)=(1-|x|-t)_+$.

Now we look at
$$v(x,t)=
\begin{cases}
u(x,t) & \hbox{if  }x\leq 0\; ,\\
1& \hbox{if  }x> 0\; .
\end{cases}
$$
Then one checks easily that $v$ is (i) \lsc, (ii) a \JVSup for \HJgen-\KC with $H_1(p)=H_2(p)=|p|$
and (iii) is not regular at $x=0$. As the reader has probably already noticed it, the \JVSup
property comes from the fact that $u$ is a solution of the state-constrained problem in $(-\infty,0]
\times (0,+\infty)$, hence $u_t + H_2(u_x)\geq 0$ even on the boundary $\{0\}\times (0,+\infty)$. 
This may be seen as a little defect of the network approach which, by using \hbox{\rm PC}$^1$-test-functions,
leads to a slight decoupling of the domains $\Omega_1 \times (0,+\infty)$ and $\Omega_2 \times (0,+\infty)$.
\end{example}

\section{Stability of junction viscosity solutions}

Contrarily to the case of flux-limited viscosity solutions, the stability result for \JVS is almost
an immediate extension of Theorem~\ref{hrl}: the change of test-functions, using $\PC1$ implies
only minor modifications which is why we skip the proof of the  \index{Junction viscosity solutions!stability}
\begin{theorem}\label{JV-stab}\emph{--- Stability of junction viscosity solutions.}\smsp
    Assume that
    \begin{enumerate}
        \item[$(i)$] For any $\e>0$, $H_1^\e, H_2^\e, G^\e$ are continuous and converge locally uniformly
        respectively to $H_1, H_2, G$\;. 
        \item[$(ii)$] For any $\e>0$, $u_{\e}$ is a \JVSub \resp{\JVSup} for the problem with Hamiltonians
        $H_1^\e, H_2^\e, G^\e$\;.
        \item[$(iii)$] The functions $u_{\e}$ are uniformly locally bounded on
        $\R^N\times[0,T_f]$.
    \end{enumerate}
    Then $\ou=\limssup u_{\e}$ \resp{$\uu=\limiinf u_{\e}$} is a \JVSub \resp{\JVSup} for the
    problem with Hamiltonians $H_1, H_2, G$.
\end{theorem}

\begin{remark}
    Contrarily to Theorem~\ref{FL-stab},  Theorem~\ref{JV-stab} turns out to be very flexible,
    without any restriction on the Hamiltonians and with general junction conditions; in particular,
    it can be used to address the problem of the convergence of the vanishing viscosity method.
\end{remark}

\section{Comparison results for junction viscosity solutions: the Lions-Souganidis approach}
\label{sec:compLS}

In this section we expose the Lions-Souganidis approach of the comparison proof for \JVS and apply
it to the case of general Kirchhoff conditions, as well as second-order equations.
\index{Comparison result!for junction viscosity solutions}\index{Kirchhoff condition}
\index{Junction viscosity solutions!comparison via the Lions-Souganidis approach}

\subsection{Preliminary lemmas}

We begin with a simple one-dimensional lemma which can be seen as a little bit more precise version of
Proposition~\ref{sub-ineq-on-b} in this context
\begin{lemma}\label{HJ-dim1}
    Let $H:\R\to\R$ be a continuous function and $u:[0,r]\to \R$ be a Lipschitz continuous 
    subsolution of $ H(u_x)=0$ in $(0,r)$. Defining
    $$ \unp :=\liminf_{x\to 0}\left[\frac{u(x)-u(0)}{x}\right]\leq
    \limsup_{x\to 0}\left[\frac{u(x)-u(0)}{x}\right]=:\ovp\;,$$
    then $H(p)\leq 0$ for all $p\in [\unp,\ovp]$.
\end{lemma}

\begin{remark}\label{rem-RT}
    In Lemma~\ref{HJ-dim1}, the subsolution is assumed to be Lipschitz continuous and this is
    consistent with the fact that we consider equations with coercive Hamiltonians, or at least 
    satisfying \NCe. This assumption ensures that $\unp$ and $\ovp$ are bounded, but this is not
    really necessary as the proof will show. 

    Without this assumption, we can still prove at least that if $\unp<+\infty$,
    $H(p)\leq 0$ for all $p\in ]\unp,\ovp[$. The importance of this remark is
    more for supersolutions: we use below an analogous result for them and it is less natural to
    assume them to be Lipschitz continuous.  
\end{remark}

\begin{proof} Let us first notice that since $u$ is assumed to be Lipschitz, both $\unp$ and $\ovp$
    are not infinite. 

    \smallskip

    \noindent\textbf{(a)} We first assume that $\unp<\ovp$.\\[2mm] 
    Let $(x_k)_k$ be a sequence of points of $(0,r)$ such that
    $x_k \to 0$ and $\big(u(x_k)-u(0)\big)/x_k \to \unp$\;.
    We pick any $\unp < p <\ovp$ and consider the function $\psi(y)=u(y)-u(0)-py$ on the interval
    $[0,x_k]$.  Since 
    $$\psi(0)=0\;,\quad \psi(x_k)<0\;,\quad \limsup_{x\to0}\frac{\psi(x)}{x}=\ovp -p>0\;,$$
    there exists a maximum point $\tilde x_k \in (0,x_k)$ of $\psi$. The subsolution property at
    $\tilde x_k$ yiels the desired inequality: $H(p)\leq 0$. Moreover, the continuity of
    $H$ implies that the same property holds true for all $p\in[\unp,\ovp]$.

    \smallskip

    \noindent\textbf{(b)} Now we turn to the case  $\unp=\ovp$.\\[2mm]
    For $0<\eps \ll 1$, we consider $v(x)=u(x) + \eps x \sin(\log(x))$. The function $x\mapsto x
    \sin(\log(x))$ is $C^1$ for $x>0$ and Lipschitz continuous. Therefore, $H(v_x)\leq o_\eps
    (1)$. Moreover 
    $$ \frac{v(x)-v(0)}{x}=\frac{u(x)-u(0)}{x} +\eps \sin(\log(x))\; ,$$
    which implies that
    $$\liminf_{x\to 0}\left[\frac{v(x)-v(0)}{x}\right]=\ovp-\eps<
    \limsup_{x\to 0}\left[\frac{v(x)-v(0)}{x}\right]=\ovp+\eps\; .$$
    Since $\ovp-\eps < \ovp < \ovp+\eps$, case \textbf{(a)} above implies that 
    $H(\ovp)\leq o_\eps (1)$ and the conclusion follows by letting $\eps$ tend to $0$.  
\end{proof}

\begin{remark}
    Of course, analogous results hold for supersolutions: if $v$ is a supersolution of $H(v_x)\geq
    0$ in $(0,r)$, it suffices to use that u=-v(x) is a subsolution of $-H(-u_x)\leq 0$ in $(0,r)$.
\end{remark}

\

In order to connect the $1$-d and multi-dimensional situations and therefore to give a more precise
formulation of Proposition~\ref{sub-ineq-on-b} in the framework which is the one of the comparison
proof, let us consider a set
$$Q:=\big\{(y,x): y\in\mathcal{V}\;,\ x\in]0,\delta[\big\}\subset\R^{p+1}$$
where $\mathcal{V}$ is a neighborhood of $0$ in $\R^p$, and $\delta>0$. If $w: \overline{Q} \to \R$, we
denote by $D^+_{\overline{Q}} w$ and $D^-_{\overline{Q}} w$ the super and sub-differentials of $w$ with respect
to both variables $(y,x)$. 

If $w$ is differentiable with respect to $y$ at $(0,0)$, it can be expected that
these sub/super-differentials of $w$ in both variables have the forms
$$(D_yw(0,0),D^-_x w(0,0))\quad\text{and}\quad (D_yw(0,0),D^+_x w(0,0))\;.$$ 
But in view of Lemma~\ref{HJ-dim1} and using Section~\ref{sect:semi.convexity}, we can give a more
precise result 
\begin{lemma}\label{sb-sp-diff}
  Let $w : \overline{Q} \to \R$ be a function such that the functions $y\mapsto w(y,x)$ are Lipschitz continuous
  in $\mathcal{V}$ uniformly  with respect to $x\in[0,\delta[$ and $y\mapsto w(y,0)$ is differentiable at $0$.
  \begin{enumerate}  
      \item[$(a)$] Superdifferential case\\[2mm] 
     We assume moreover that $w$ is upper-semicontinuous in $Q$ and that, for any
     $x\in  [0,\delta[$, the function $y\mapsto w(y,x)$ is semi-convex in
     $\mathcal{V}$. If 
     $$\overline{p} = \limsup_{x\to  0}\left[\frac{w(0,x)-w(0,0)}{x}\right]$$
     exists and is finite, then
     $$ (D_y  w(0,0), p) \in D_Q^+ w(0,0)\quad\hbox{if and only if}\quad p \geq \overline{p}\;.$$

  \item[$(b)$] Subdifferential case\\[2mm]
     We assume moreover that $w$ is lower-semicontinuous in $Q$ and that, for any
     $x\in  [0,\delta[$, the function $y\mapsto w(y,x)$ is semi-concave in
     $\mathcal{V}$. If 
     $$\underline{q} = \liminf_{x\to  0}\left[\frac{w(0,x)-w(0,0)}{x}\right]$$
     exists and is finite, then
   $$ (D_y  w(0,0), q) \in D_Q^- w(0,0)\quad\hbox{if and only if}\quad q \leq \underline{q}\ \;.$$
  \end{enumerate}
\end{lemma}

The interest of this lemma is clear:  under suitable assumptions, we can connect
$1$-d and multi-d sub or super-differentials. This will be a key step for
applying Lemma~\ref{HJ-dim1} to multi-d problems. We point out anyway that 
Lemma~\ref{HJ-dim1} gives an important additional information on the interval $[\unp,\ovp]$ if $\unp \neq \ovp$.

\begin{proof}
    We only do the proof in case $(a)$, the other case working with obvious
    adaptations. 
    
    For $p\geq\bar p$, we set 
    $$\overline w(y,x):= w(y,x)-w(0,0)-D_y  w(0,0)\cdot y-px\;.$$
    If $(D_y  w(0,0), p) \in D_Q^+ w(0,0)$, then $ \overline w(y,x) \leq o(|y|+x)$.
    Choosing $y=0$ and dividing by $x>0$, we obtain
    $$ \frac{w(0,x)-w(0,0)}x-p \leq o(1)\; ,$$
    and therefore $\overline{p} \leq p$ by taking the $\limsup$ as $x\to 0$.
    
    Conversely, if $p\geq \overline{p}$, we want to show that $(D_y  w(0,0), p) \in D_Q^+ w(0,0)$, \ie $\overline w(y,x) \leq o(|y|+x)$. 
    To do so, we argue by contradiction assuming that there exists $\eta >0$ and a sequence $(y_k,x_k)_k$ converging to 
    $(0,0)$ such that $x_k>0$ for all $k$ $\overline w(y_k,x_k) \geq \eta(|y_k|+x_k)$.
    
    Using the upper semicontinuity of $w$, hence of $\overline w$, we easily deduce that $\overline w(y_k,x_k)\to \overline w(0,0)=0$
    and the Lipschitz continuity of $y\mapsto \overline w(y,x_k)$ implies that $\overline w(0,x_k)\to \overline w(0,0)=0$.
    
    Next we use the decomposition 
    $$ \overline w(y,x_k)= [\overline w(y,x_k)-\overline w(0,x_k)]+ \overline w(0,x_k)\; .$$
    By the definition of $\bar p$ and $p\geq \bar p$, 
    $\overline w(0,x_k)=w(0,x_k)-w(0,0)-px_k\leq o(x_k)$. Therefore it remains to estimate the bracket to obtain a contradiction.

To do so, we introduce a regularization by convolution of $y\mapsto \overline w(y,x_k)$ for all fixed $x_k$. Let $(\rho_\eps)_\eps$
be a sequence of approximate identities in $\R^p$,
$i.e.$ a sequence of positive, $C^\infty$-functions on $\R^p$ with compact
support in $B_\infty(0,\eps)$ such that $\int_{\R^p}\,\rho_\eps(z)dz=1$. 
Then we set 
$$ \overline w_\e
(y,x_k):= \int_{\R^p} \overline w(y+z,x_k)\rho_\eps(z)\,dz\; .$$ 
For any $k$, the functions $y\mapsto \overline w_\e (y,x_k)$ are $C^1$ in a neighborhood of $0$
and therefore 
$$\overline w_\e (y,x_k)-\overline w_\e (0,x_k)= \int_0^1 D_y\overline w_\e
(s y,x_k)\cdot y\, ds\; .$$ 
Now we examine $D_y\overline w_\e (s y,x_k)$.
Since, for any $k$, 
$y\mapsto \overline w(y,x_k)$ is Lipschitz continuous, hence differentiable almost everywhere by Rademacher's Theorem,
we have  
$$D_y \overline w_\e (s y,x_k) =\int_{\R^p} D_y \overline w(s y+z,x_k)\rho_\eps(z)\,dz \; .$$
Moreover, using again the Lipschitz continuity of $y\mapsto \overline w(y,x_k)$, we have $\overline w(y,x_k)\to \overline w(0,0)=0$
when $y\to 0$ and $k\to +\infty$; therefore, by the semi-convexity assumption, it follows from 
Proposition~\ref{prop:semi-conv-prop}-$(iv)$ that we
have $D_y \overline w(y,x_k)=D_yw(y,x_k)-D_yw(0,0)\to 0$ when $(y,x_k) \to (0,0)$. This
implies that
$D_y \overline w_\e (s y,x_k)=o_k(1)+o_y(1)+o_\e(1)$ as $(y,x_k,\e)\to (0,0,0)$, 
uniformly with respect to $s\in[0,1]$.
Therefore $\overline w_\e (y,x_k)-\overline w_\e (0,x_k)=|y|(o_x(1)+o_y(1)+o_\e(1))$. 
Letting $\e$ tend to $0$, we end up with
$$\overline w(y,x_k)-\overline w(0,x_k)=o(|y|)+|y|o_x(1)=o(|y|)\quad \hbox{for  $k$ large enough},$$
which yields the desired contradiction.  
\end{proof}

\subsection{A comparison result for the Kirchhoff condition}

\index{Kirchhoff condition}
Before considering other junction conditions, 
we first provide a comparison result for the problem \HJgen-\KC, namely
$$\begin{cases} u_t + H_1(x,t,u,Du) = 0 & \hbox{in  }\Omega_1 \times (0,\Tf) \; ,\\[2mm]
    u_t + H_2 (x,t,u,Du) = 0 & \hbox{in  }\Omega_2 \times (0,\Tf)\; ,\\[2mm]
    \dfrac{\partial u}{\partial n_1}+\dfrac{\partial u}{\partial n_2}=0 & 
\hbox{on  }\H\times (0,\Tf)\;,\end{cases}$$
where for $i=1,2$, $n_i(x)$ denotes the unit outward to $\domeg_i$ at $x\in \domeg_i$. 
We recall that this Kirchhoff condition has to be taken in the \JVS sense, namely
$$ \min\Big(u_t + H_1(x,t,u,Du),u_t + H_2 (x,t,u,Du),\frac{\partial u}{\partial n_1}+
\frac{\partial u}{\partial n_2}\Big)\leq 0\quad \hbox{on  }\H\ \times (0,\Tf)\;,$$
for the subsolution condition and 
$$ \max\Big(u_t + H_1(x,t,u,Du),u_t + H_2 (x,t,u,Du),\frac{\partial u}{\partial n_1}+
\frac{\partial u}{\partial n_2}\Big)\geq 0\quad \hbox{on  }\H\times(0,\Tf)\;,$$
for the supersolution condition, using test-functions in $\mathrm{PC}^{1}$.

In order to formulate and prove a comparison result for \HJgen-\KC, we face several difficulties:
first, we are not readily in the ``good framework for HJ-Equations with discontinuities'' because of
the $\min$ in the subsolution junction condition and the \KC condition which prevents \NCe to be
satisfied. As a consequence it is not clear a priori that subsolutions are regular---a general
problem for \JVS---nor supersolutions. This last point is important since, in order to apply the
Lions-Souganidis approach, we have to regularize both the sub and supersolution, needing both
to be regular. Fortunately for subsolutions, this problem is solved by Proposition~\ref{prop:JVreg},
but not for supersolutions. Last but not least, it is not completely clear that we can apply
Proposition~\ref{reducCR} in order to prove only \LCR.

We can overcome all these difficulties under some suitable assumptions
\begin{theorem}\label{LS-CR}\emph{--- Comparison result, the Kirchhoff case.}\smsp
    Assume that $H_1$ and $H_2$ satisfy \GAGen. Then the \GCR holds for any bounded subsolution $u$ 
    and supersolution $v$ provided\\[2mm]
    $(i)$ either $v$ satisfies \eqref{Lip-z-v}, \ie there exists $C>0$ such that for all $(x',t)$
    $$|v((x',x_N),t)-v(x',y_N),t)|\leq C|x_N-y_N|\,;$$
    $(ii)$ or \TCsH holds for both $H_1$ and $H_2$.
\end{theorem}

Of course, the main interest of this result is to allow to prove a \GCR which is valid for non
convex Hamiltonians $H_1$ and $H_2$. In addition, it is easy to see that the proof we give below
(and which is almost exactly the Lions-Souganidis one) can provide a comparison result for different
types of ``junction conditions'' on $\H$ and also for more general networks problems; we come back
on this point in the next section.

\begin{proof}
    This proof consists first in reducing to a one-dimensional proof thanks to various
    reductions and using the preliminary lemmas of the previous section.

    \medskip

    \noindent\textbf{(a)} \emph{Reduction to a \LCR with semiconvex/concave functions.}\\[2mm]
    Thanks to Section~\ref{sect:htc}, we are not going to prove a \GCR but only a \LCR: the results
    of this section apply since the modifications we perform on the subsolution $u$ are $C^1$ and
    therefore do not affect the \KC condition. The next point concerns the regularity of $u$ and $v$
    on $\H$: both are regular by Proposition~\ref{prop:JVreg}. 

    Now, proving the \LCR means that if $u$ is a subsolution and $v$ is a supersolution of \HJgen-\KC, 
    we want to prove that there exists $r>0$ and $0<h<t$ such that, denoting by 
    $\mK:=\overline{Q^{x,t}_{r,h}}$, if $\max_\mK(u-v) >0$, then
    $$ \max_\mK(u-v) = \max_{\partial_p \mK}(u-v)\;.$$
    Considering a point point $(\xb,\tb)$ where $\max_\mK(u-v)>0$ is achieved, we can assume of
    course that $\tb>0$ and $(\xb,\tb) \notin \partial_p \mK$ otherwise the result obviously holds. 
    It is also clear that we can assume \wlg that $\xb\in \H$, otherwise only the $H_1$ or $H_2$
    equation plays a role and we are in the case of a standard proof.

    The proof of the theorem is based on the arguments of Section~\ref{sect:sup.reg}, and more
    precisely on Propositions~\ref{reg-by-sc} and \ref{reg-by-ic}: by using
    Proposition~\ref{reg-by-sc} and adding to $u$ a term of the form $\eta \chi(x_N)$ where $\chi
    \in\mathrm{PC}^1(\R)$ is a bounded function which satisfies $\chi'(0^+)=1$ and $\chi'(0^-)=-1$,
    we can assume \wlg that $u$ is a Lipschitz continuous, $\eta$-strict subsolution of the equation
    and that $u$ is semi-convex in $x'$ and $t$. 

    Using similar arguments, it is also possible to assume that $v$ is semi-concave
    in  $x'$ and $t$ but only under the conditions of Proposition~\ref{reg-by-ic}, hence assumptions
    $(i)$ or $(ii)$ above. We point out that
    these reductions allow to have a supersolution $v$ which is Lipschitz continuous in $x'$ and
    $t$, uniformly in $x_N$, but which can still be discontinuous in $x_N$, we come back on
    this point below by modifying $v$ into some supersolution $\tilde v$.

    A key consequence of the semi-convexity of $u$ and of the semi-concavity of $v$ in the variables
    $x',t$ is that $u$, $v$ are differentiable in $x'$ and $t$ at the maximum point $(\xb,\tb)$ and
    $$ D_{x'}u(\xb,\tb)=D_{x'} v(\xb,\tb)\quad \hbox{and}\quad u_t (\xb,\tb)=v_t (\xb,\tb)\; .$$
    For a precise result, see Proposition~\ref{prop:semi-conv-prop}-$(v)$.
    Moreover, as a consequence of Remark~\ref{rem:cont-SC-D} (since the semi-convexity of $u$
    holds only in the tangent variables), if we denote by $(p',p_N,p_t)$ 
    any element in the superdifferential of $u$ at $(y,s)$
    close to $ (\xb,\tb)$, then $(p',p_t) \to (D_{x'} u(\xb,\tb),u_t (\xb,\tb))$ as $(y,s)$ tends to
    $ (\xb,\tb)$.

    For the supersolution $v$ however, the same property may not be true for the elements of
    the subdifferential since $v$ can be discontinuous at $ (\xb,\tb)$. To turn around this
    difficulty we introduce 
    $$ \tilde v ((x',x_N),t):= \min\Big(v ((x',x_N),t)\;,\, v ((x',0),t)+K|x_N| \Big)\;,$$
    where $K>0$. If we choose $K$ large enough, function $\tilde v$ is a supersolution of the
    equation for $x_N\neq 0$, as the minimum of two supersolutions and is continuous at $ (\xb,\tb)$.
    As a consequence of this continuity property, $\tilde v$ being still semi-concave in $(x',t)$ as
    the minimum of semi-concave functions in $(x',t)$, for any element $(p',p_N,p_t)$ in the
    subdifferential of $\tilde v$ at $(y,s)$ close to $ (\xb,\tb)$, the following limit holds:
    $(p',p_t) \to (D_{x'} \tilde v(\xb,\tb),\tilde v_t (\xb,\tb))=(D_{x'}
    v(\xb,\tb),v_t (\xb,\tb))$ as $(y,s)$ tends to $ (\xb,\tb)$.

    \medskip

    \noindent\textbf{(c)} \emph{Reduction to a stationary, one-dimensional proof.}\\[2mm]
    These properties of $u$ and $\tilde v$ allow us to argue only in the $x_N$ variable since,
    taking into account the regularity of $H_1,H_2$, we have 
    $$ \tilde H_1 (u_{x_N} ) \leq -\eta <0\leq  \tilde H_1 (\tilde v_{x_N} )
    \quad \hbox{for  }x_N >0\; ,$$
    $$ \tilde H_2 (u_{x_N} ) \leq -\eta < 0\leq  \tilde H_2 (\tilde v_{x_N} )
    \quad \hbox{for  }x_N <0\; ,$$
    where for $i=1,2$,
    $$ \tilde H_i (p_N)= u_t (\xb,\tb)+H_i(\xb,\tb,u (\xb,\tb),(D_{x'} u(\xb,\tb),p_N))+o(1)\;,$$
    the $o(1)$ tending to 0 as $\bar r\to 0$ if we consider the equations in $B((\xb,\tb),\bar r)$. 
    It is worth pointing out that for the $H_i$-equations for $v$, we have used the fact that both
    $r\mapsto H_i(x,t,r,p)$ are increasing and that $u(\xb,\tb)>\tilde v(\xb,\tb)=v(\xb,\tb)$.

    In order to proceed, we compute the superdifferentials for $u$ in the two directions $x_N>0$ and
    $x_N<0$.  We recall that since the test-functions are different in $\Omega_1$ and $\Omega_2$,
    these superdifferentials are different.
    For $x_N >0$, we have $D^+_1 u(0)=[\ovp_1,+\infty)$ where $\ovp_1$ is defined as the $\ovp$ in
    Lemma~\ref{HJ-dim1}, but we are referring here to $\Omega_1$.  For $x_N<0$, we have $ D^+_2
    u(0)=[-\infty,-\ovp_2)$ where $\ovp_2$ is defined as the $\ovp$ in Lemma~\ref{HJ-dim1} but for
    $u(-x_N)$, in $\Omega_2$.

    Using the definition of viscosity subsolution together with Lemma~\ref{HJ-dim1}
    and~\ref{sb-sp-diff}, we obtain, since $n_1=-e_N$ and $n_2=e_N$ 
    $$\min(-p_1+p_2, \tilde H_1 (p_1) + \eta, \tilde H_2 (p_2 ) + \eta)\leq 0\; ,$$
    for any $p_1 \geq \ovp_1$ and $p_2\leq - \ovp_2$; moreover
    $$\begin{aligned}
    \tilde H_1 (p_1 ) + \eta & \leq 0 \quad \hbox{if  }p_1 \in [\unp_1,\ovp_1]\; ,\\
    \tilde H_2 (p_2 ) + \eta & \leq 0 \quad \hbox{if  }p_2 \in [-\ovp_2, -\unp_2]\; .
    \end{aligned}$$

    For the supersolution $v$, we argue through $\tilde v$ but the aim is really to identify the
    subdifferential of $v$ at $(\xb,\tb)$. We first notice that, $(\xb,\tb)$ being a maximum point
    of $u-v$, then $ u(x,t)-v(x,t)\leq u(\xb,\tb)-v(\xb,\tb)$ for any $(x,t)$ and the Lipschitz
    continuity of $u$ implies
   $$ -C|(x,t)-(\xb,\tb)| \leq u(x,t)-u(\xb,\tb) \leq v(x,t)-v(\xb,\tb)\;,$$ 
    for $C$ large enough, and in particular in the $x_N$-direction 
\begin{equation}\label{eq:cons-max-prop}
 -C|x_N| \leq v((\xb',x_N),\tb)-v(\xb,\tb)\;.
\end{equation}
 Hence the subdifferentials of $v$ at $(\xb,\tb)$ in both directions, namely $D^-_1 v(0)$ and 
    $D^-_2 v(0)$ are non empty.

    Moreover, applying Lemma~\ref{HJ-dim1} to $\tilde v$, we
    obtain that 
    $$\begin{aligned}
    \tilde H_1 (q_1 ) & \geq 0 \quad \hbox{if  }q_1 \in [\unq_1,\ovq_1]\;,\\
    \tilde H_2 (q_2 ) & \geq 0 \quad \hbox{if  }q_2 \in [-\ovq_2,-\unq_2]\;,
    \end{aligned}$$
    where $D^-_1 \tilde v(0)=(-\infty,\unq_1]$ and $ D^-_2 \tilde v(0)=[-\unq_2,+\infty)$. 

    On the other hand, using Lemma~\ref{sb-sp-diff}, $(D_{x'} v(\xb,\tb),q_1,v_t
    (\xb,\tb)) \in D^-_{\overline{Q_1}^\ell} \tilde v(\xb,\tb)$ \footnote{We recall
    that $D^-_{\overline{Q_1}^\ell} \tilde v(\xb,\tb)$ denotes the
    subdifferential {\em related to $\overline{Q_1}^\ell$} of the function
    $\tilde v$ at $(\xb,\tb)$} for any $q_1 \leq \unq_1$.

    In order to connect the sub-differentials of $v$ and $\tilde v$, we use 
    the following classical result whose proof is an exercise left to the reader.
    \begin{lemma}\label{subdiff-A}
        Let $w_1,w_2:A\subset\R^p\to\R$ be two \lsc functions such that
        $w_1(z_0)=w_2(z_0)$ for some $z_0\in A$. Then
        $$D^-_A \min(w_1,w_2)(z_0)\subset D^-_A w_1(z_0) \cap D^-_A w_2(z_0)\;.$$
    \end{lemma}

    Applying the result with $A:=\Omegb_1\times [0,\Tf] $, $z_0:=(\xb,\tb)$, $w_1(x,t)=
    v(x,t)$ and $w_2(x,t)=v ((x',0),t)+ K|x_N|$, we deduce that $(D_{x'} v(\xb,\tb),q_1,v_t
    (\xb,\tb)) \in D^-_{\overline{Q_1}^\ell} v(\xb,\tb)$.  Of course, the same arguments can be used
    for $D^-_{\overline{Q_2}^\ell} v(\xb,\tb)$.

    Hence, for any $q_1 \leq \unq_1$ and $q_2 \geq - \unq_2$,
    $$\max(-q_1+q_2, \tilde H_1 (q_1 ) , \tilde H_2 (q_2 ))\geq 0\;.$$

    \medskip

    \noindent\textbf{(c)} \emph{Using the viscosity inequalities to get contradictions.}

    \medskip

    \noindent{\bf Case 1 :} Either $[\unp_1,\ovp_1]\cap  [\unq_1,\ovq_1] \neq \emptyset$ or
    $[-\ovp_2, -\unp_2] \cap [-\ovq_2,-\unq_2]\neq \emptyset$ : this means that there exists $p$
    such that we have $$\text{either } \tilde H_1 (p ) + \eta \leq 0 \leq \tilde H_1 (p)\;,\;
    \text{or } \tilde H_2 (p ) + \eta \leq 0 \leq \tilde H_2 (p)\; ,$$ and in each case we reach a
    contradiction.

    \medskip

    \noindent{\bf Case 2 :} Otherwise, since $0$ is a maximum point of $u-v$, we have necessarily
    $\ovp_1\leq \ovq_1$ and therefore $\unp_1 \leq \ovp_1 < \unq_1 \leq \ovq_1$. Considering the
    function $p\mapsto \tilde H_1(p)$ which is less that $-\eta$ in $[\unp_1,\ovp_1]$ and positive
    in $[\unq_1,\ovq_1]$, we see that there exists $\ovp_1 < r_1< \unq_1 $ such that $\tilde H_1
    (r_1)=-\eta/2$.

    Similarly, $-\ovq_2 \leq - \unq_2<-\ovp_2\leq -\unp_2$ and there exists 
    $ - \unq_2<r_2 < -\ovp_2$ such that $\tilde H_2 (r_2)=-\eta/2$.  Then, choosing $\delta>0$ small
    enough and $p_1= r_1-\delta$, $p_2 = r_2 +\delta$, we have $p_1 \geq \ovp_1$ and $p_2\leq -
    \ovp_2$.  Therefore the viscosity inequalities give $$\min(-p_1+p_2, \tilde H_1 (p_1 ) + \eta,
    \tilde H_2 (p_2 ) + \eta)\leq 0\; ,$$ but with the choice of $\delta$, $\tilde H_1(p_1) +
    \eta>0$, $\tilde H_2(p_2) + \eta>0$, which implies $-p_1+p_2 \leq 0$, in other words
    $-r_1+r_2+2\delta\leq0$.

    On the other hand, choosing $q_1= r_1+\delta$ and $q_2 = r_2 -\delta$ and using $ \tilde H_1
    (q_1) <0$, $\tilde H_2 (q_2 ) <0$, we also get $-q_1+q_2 \geq 0$ which leads to a contradiction
    because this implies $-r_1+r_2 -2\delta \geq 0$.

    The conclusion is that $\max_\mK(u-v)$, if positive, cannot be reached inside $\mK$ but
    necessarily on $\partial_P\mK$, which ends the proof\footnote{The authors wish to
    thank Peter Morfe for pointing out several unclear points in this proof which led us to several
    improvements, in particular the statements of Lemma~\ref{sb-sp-diff} and~\ref{subdiff-A}.}.
\end{proof}

\subsection{Remarks on the comparison proof and some possible variations}
\label{rem:v-nonregul} 

In the above proof, the following points are crucial\\

(a) Because of the ``normal controllability assumption'', the tangential regularization of the subsolution does not cause any problem
and provides us with a Lipschitz continuous subsolution to which the Lions-Souganidis Lemma (Lemma~\ref{HJ-dim1}) fully applies:
as a consequence, we can argue as if the problem was $1$-dimensional and we obtain informations on the equation and the junction condition
not only for the elements of the superdifferentials  $D^+_1 u(0)=[\ovp_1,+\infty)$ and $ D^+_2
    u(0)=[-\infty,-\ovp_2)$,  but also on $\tilde H_1 (p_1 )$ for $p_1 \in [\unp_1,\ovp_1]$ and $\tilde H_2 (p_2 )$ for $p_2 \in [-\ovp_2,\unp_2]$, \ie a priori  on larger intervals than the expected ones. \\

(b) The situation is completely different for the supersolution, mainly because the ``normal controllability assumption'' cannot play the
same role. This first generates Assumptions $(i)$ or $(ii)$ in order to be able to do the tangential regularization of the supersolution but
we end up with a function which is Lipschitz continuous and semi-concave in the tangential variables but which can still be discontinuous
in $x_N$. In particular at $x_N=0$. This is a difficulty to apply the Lions-Souganidis approach because we cannot prove that
$\tilde H_i (v_{x_N} )\geq 0$ for any $x_N$ since we cannot use Proposition~\ref{prop:semi-conv-prop}-$(iv)$ which provides some
kind of "continuity" for the tangential derivatives but only if we have the right continuity property on $v$. In the above proof, we have chosen a strategy which consists in fully applying Lions-Souganidis Lemma to the supersolution but this requires the introduction of $\tilde v$. With this trick, whose aim is to superimpose
the continuity of the supersolution at $x_N=0$, the case of the supersolution can be treated as the subsolution one.\\

(c) Is this trick necessary/unavoidable? This question is important to extend the result of Theorem~\ref{LS-CR} to more general junction 
conditions and in particular to \FL or \GJC of flux-limited types for which such trick may not work. A first natural reaction could be to accept
the discontinuity in $x_N$ of the regularized supersolution and to try to prove Lemma~\ref{HJ-dim1} in this framework. This
may be feasible but a natural assumption to do that would be (at least) to have a regular supersolution, an assumption that we would like
to avoid. Before going further, we remark that the proof of Theorem~\ref{LS-CR} consists in examining carefully the properties of the
elements of the different super and subdifferentials of the sub and supersolution (and even more for the subsolution as we mention it above). At this point, it is worth mentioning that the maximum point property \eqref{eq:cons-max-prop} yields
\begin{equation}\label{eq:cons-max-prop-bis}
u(x_N)-u(0) \leq v(x_N)-v(0)\; ,
\end{equation} 
and the first consequence of this inequality and of the Lipschitz continuity of $u$ is that $D^-_1 v(0)$ and $ D^-_2 v(0)$ are
non-empty. Hence we have, for any $q_1 \in D^-_1 v(0)$ and $q_2 \in D^-_2 v(0)$
\begin{equation}\label{eqn:visc-ineq-sup-bound}
\max(-q_1+q_2, \tilde H_1 (q_1 ) , \tilde H_2 (q_2 ))\geq 0\;.
\end{equation}
Then we argue in the following way
\begin{enumerate}
\item If $D^-_1 v(0)=(-\infty, \unq_1]$, then \eqref{eqn:visc-ineq-sup-bound} holds but Lemma~\ref{sb-sp-diff} applies and gives, in
addition, $\tilde H_1(\unq_1)\geq 0$. And, of course, we may use a similar argument if $ D^-_2 \tilde v(0)=[-\unq_2,+\infty)$.

\item Or $D^-_1 v(0)=\R$; this is the case, in particular, if $v$ is not $\Omega_1\times (0,\Tf)$-regular at $(\xb,\tb)$ and \eqref{eqn:visc-ineq-sup-bound} holds for any $q_1 \in \R$ and $q_2 \in D^-_2 v(0)$. In addition, $\tilde H_1 (q_1 )\geq 0$ if $q_1 \geq \unq_1$ for some $\unq_1$ by the coercivity of $\tilde H_1$. And, of course, a similar argument holds if 
$D^-_2 v(0)=\R$.
\end{enumerate}
In summary, by applying Lemma~\ref{sb-sp-diff} instead of Lemma~\ref{HJ-dim1}, not only we have more informations than we could have
obtained by applying the Lions-Souganidis Lemma but we can also avoid introducing $\tilde v$. \\

(d) Last very important point: in Equation~\ref{eq:cons-max-prop-bis}, for $x_N>0$,
we can divide by $x_N$ and take the $\liminf$: by using Lemma~\ref{HJ-dim1} for $u$ and Lemma~\ref{sb-sp-diff} for $v$,
we obtain the inequality $\unp_1 \leq \unq_1$ where the  information on $\unp_1$ becomes crucial (we again insist on the fact that $
\unp_1 \notin D_1^+u(0)$).\\

With all these ingredients, we will be able in the next section to extend the result of Theorem~\ref{LS-CR} to a large class of \GJC without
assuming any regularity on the supersolution.

\subsection{Comparison results for more general junction conditions}

Theorem~\ref{LS-CR} can be generalized for \FL and \GJC conditions under some hypotheses:
\begin{theorem}\label{LS-CR-GC}\emph{--- Comparison result for general junction conditions.}\smsp
    Assume that $H_1$ and $H_2$ satisfy \GAGen and \TCsH. Then \GCR holds 
    \begin{enumerate}
    \item in the case of \FL: for any bounded regular subsolution $u$ and supersolution $v$ if $G$
    satisfies \GAGFL.  
    \item in the case of \GJC of ``Kirchhoff type'', \ie if $G$ satisfies
    \GAGKT: for any bounded subsolution $u$ and supersolution $v$.
    \item in the case of \GJC of ``Flux-limited type'', \ie if $G$ satisfies \GAGFLT: for any bounded
    regular subsolution $u$ and supersolution $v$.  
    \end{enumerate}
    Moreover, if $v$ is locally Lipschitz continuous in $x_N$, uniformly in $x',t$, then these three
    results hold true without assuming \TCsH for $H_1$, $H_2$ and, in the last case by assuming only
    that $G$ satisfies \HContG with $\e_0=1$ and \eqref{basic-hyp-GJC} holds with $\alpha >
    0$, $\beta=0$.
\end{theorem}

Several remarks can be made on Theorem~\ref{LS-CR-GC}. First, for \FL type conditions, we face the
difficulty that sub and supersolutions may not be regular and we have to add these properties as an
assumption. 

But the main one follows along the lines of the remark we made when introducing the
assumptions \GAGFL, \GAGKT, \GAGFLT. Using the Lions-Souganidis approach as in the proof of 
Theorem~\ref{LS-CR} requires a ``tangential regularization'' both for the sub and the supersolution in the
spirit of Sections~\ref{sec:regplus-subsol} and~\ref{sec:regul-supersol} in order to reduce to a
$1$-dimensional proof. While this regularization does not cause much problem in the first case of \FL
---except that we have to impose \TCsH for $H_1$ and $H_2$---, it requires a particular treatment when
$G$ is of ``Kirchhoff type'',  and a particular form of $G$ when it is of ``flux-limited type''. For
this reason, we need \GAGFLT\  which is roughly speaking the analogue of \TCs.

\begin{proof} 
    We are not going to give the full proof of Theorem~\ref{LS-CR-GC} since most of the arguments
    are those of the proof of Theorem~\ref{LS-CR}, using in an essential way the Lions-Souganidis
    approach. We just indicate the additional arguments which are needed.

    We first comment \LOCa,\LOCb: in the case of \FL or when \GJC is of ``Flux-limited type'', the
    checking can be made exactly as in Section~\ref{htc:ev}, the $u_t$-term playing the main role.
    In the case when \GJC is of ``Kirchhoff type'', one has to add a $\alpha \varphi(x_N)$ (or
    $\delta \varphi(x_N)$) term where $\varphi \in \mathrm{PC}^{1} (\R)$ is a bounded function which
    behaves like $L|x_N|$ in a neighborhood of $x_N=0$. With this additional argument, \LOCa,\LOCb
    hold in the three cases and we can reduce the proof to a \LCR.

    The next important point concerns the regularization of both the sub and supersolution which we
    examine case by case.

    \medskip

    \noindent\textbf{(a)} \FL case --- the regularity of the sub and supersolution is a key assumption
    since we have seen that \JVSub and \JVSup are not necessarily regular but the rest of the proof
    follows readily the arguments of Sections~\ref{sec:regplus-subsol} and~\ref{sec:regul-supersol}
    because \TCs is ensured by \TCsH or \GAGFL, even if, for the subsolution, the equation does not
    satisfy \NCe on $\H$.

    \medskip

    \noindent\textbf{(b)} ``Flux-limited type'' \GJC --- The same comments are true here and explain the
    restrictive assumptions we have to impose in this case.

    \medskip    

    \noindent\textbf{(c)} ``Kirchhoff type'' \GJC --- This is the most complicated case where we need the

    \begin{lemma}\label{reg-GKC}
        Under the assumptions of Theorem~\ref{LS-CR-GC} in the case of \GJC of ``Kirchhoff type'',
        if $u$ and $v$ are respectively a sub and supersolution of \HJgen-\GJC, then, for $K_2$
        large enough and $K_1$ large compared to $K_2$, the functions 
        $$u^{\e}(x',x_N,t):=\max_{(y',s) }
        \Big\{u(y',x_N,s)-\exp(K_1t)\exp(-K_2|x_N|)\big( \frac{|x'-y'|^2}{\e^2}+
        \frac{|t-s|^2}{\e^2}\big)\Big\}\;,$$
        $$v^{\e}(x',x_N,t):=\min_{(y',s) }
        \Big\{v(y',x_N,s)+\exp(K_1t)\exp(-K_2|x_N|)\big( \frac{|x'-y'|^2}{\e^2}+
        \frac{|t-s|^2}{\e^2}\big)\Big\}\;,$$
        are respectively approximate \JVSub and \JVSup of \HJgen-\GJC.
    \end{lemma}

\begin{proof}[Proof of Lemma~\ref{reg-GKC}] 
    The proof follows essentially the proofs of the regularization results of
    Sections~\ref{sec:regplus-subsol} and~\ref{sec:regul-supersol} with the following additional
    arguments: first, we notice that we can use a $\mathrm{PC}^{1}$-term $\exp(-K_2|x_N|)$.  If
    $x_N>0$ or $x_N<0$, this term produces a ``bad term'' in the $x_N$-derivative which has to be
    controlled by the $t$-derivative coming from the $\exp(K_1t)$-term.

    At $x_N=0$, if the $G$-inequality holds, the $u_t$-term cannot control the ``bad
    terms'' anymore but the $\exp(K_1t)$-term has the ``good sign'' and, since $\beta>0$, the
    derivatives of $\exp(-K_2|x_N|)$ allow to control all the error terms.
\end{proof}

    Once this regularization is done, we can in each case transform the approximate subsolution into
    a strict subsolution using either a $-\eta t$ term or a $\eta\varphi(x_N)$ term where $\varphi$
    is the function we already used at the beginning of the proof.

    From this point, we can follow readily the proof of Theorem~\ref{LS-CR} to conclude.
\end{proof}

\subsection{Extension to second-order problems (II)}

The same approach allows to deal with second-order problems, with similar structure assumptions on
the Hamiltonians.  
\begin{theorem}\emph{--- Comparison in the second-order case, LS-version.}\label{comp-IM-so-LS}\smsp
    Under the assumptions of Theorem~\ref{LS-CR-GC}, the comparison result remains valid for
    Lipschitz continuous sub and supersolutions of second-order equations of the form
    \begin{equation}\label{so-eqn-LS}
    u_t-{\rm Tr}(a_i(x)D^2u)+ H_i(x,t,u,Du)=0\quad \hbox{in  }\Omega_i \times (0,\Tf)\;,
    \end{equation}
    provided that, for $i=1,2$, $a_i=\sigma_i\sigma_i^T$ for a bounded, Lipschitz continuous
    function $\sigma_i$, depending only on $x_N$ in a neighborhood of $\H$. 
\end{theorem}

Of course, the most restrictive assumption in Theorem~\ref{comp-IM-so-LS} is the Lipschitz
continuity of the sub and supersolutions to be compared. But if we examine the proof in the
first-order case, we remark that the tangential regularization provides a Lipschitz continuous
subsolution because of the normal controllability and the regularized supersolution ``behaves'' like
a Lipschitz continuous function because of the maximum point property in the proof of the \LCR. All
these arguments, and in particular the first one, fail here because of the second-order term and we
need to replace them by the \adhoc assumption. 

    The proof follows the arguments of the proof of Theorem~\ref{LS-CR-GC}, except that we need the
    following extension of the Lions-Souganidis Lemma (Lemma~\ref{HJ-dim1}).

\begin{lemma}\label{HJ-dim1-so}
    We assume that for some $r>0$, $u:B(0,r)\times [0,r] \subset \R^{N-1} \times \R \to \R$ is an \usc
    subsolution \resp{\lsc supersolution} of 
    $$ -{\rm Tr}(a(x_N)D^2w)+ \tilde H(w_{x_N})=0\quad\hbox{in  }B(0,r)\times (0,r)\;,$$
    where $\tilde H$ is a continuous function and $a(x_N)=\sigma(x_N)\sigma^T(x_N)$ for some bounded
    Lipschitz continuous function $\sigma$. If moreover
    \begin{enumerate} 
        \item[$(i)$] $u(x',x_N)$ is Lipschitz continuous and semi-convex \resp{semi-concave} in $x'$, 
        uniformly for $x_N\in [0,r]$\;,
    \item[$(ii)$] there exists a constant $k>0$, such that\quad $u(0,x_N)-u(0,0)\leq k x_N$\\[2mm]
        \resp{$u(0,x_N)-u(0,0)\geq k x_N$} for $0\leq x_N\leq r$\;,
    \item[$(iii)$] the function $u$ is differentiable w.r.t. $x'$ at $(0,0)$\;,
    \end{enumerate}
    then $\tilde H(p)\leq 0$ \resp{$\tilde  H(p)\geq 0$} for all $p\in [\unp,\ovp]$ where
    $$\unp :=\liminf_{x_N\to 0}\left[\frac{u(0,x_N)-u(0,0)}{x_N}\right]\leq
    \ovp:=\limsup_{x_N\to 0}\left[\frac{u(0,x_N)-u(0,0)}{x_N}\right]\;.$$
\end{lemma}

    We first point out that the assumption on $u$ implies that it is continuous at $(0,0)$ but may
    still have discontinuities outside $(0,0)$.

    The additional difficulty in this lemma (compared to Lemma~\ref{HJ-dim1}) is the $x'\in
    \R^{N-1}$-dependence in the $D^2 u$-term which cannot be dropped by the semi-convex (or
    semi-concave) assumption on $u$.

\begin{proof} 
    We only give the proof in the subsolution case, the supersolution one being analogous.

    Replacing $u(x',x_N)$ by the subsolution $u(x',x_N)-u(0,0)-D_{x'} u(0,0)\cdot
    x'$, we can assume \wlg that $u(0,0)=0$ and $D_{x'} u(0,0)=0$. Also, as in the proof of
    Lemma~\ref{HJ-dim1}, we first assume that $\unp<\ovp$.

    \medskip

    \noindent\textbf{(a)} 
    The first step consists in giving some estimates on $u$ for $(x',x_N)$ close to
    $(0,0)$. By the semi-convexity assumption, if $u$ is differentiable in $x'$ at $(x',x_N)$ and if
    $(x',x_N)$ is close to $(0,0)$, then $|D_{x'} u(x',x_N)|$ is also close to $|D_{x'} u(0,0)|= 0$ and
    we deduce from this property that
    $$ u(x',x_N)-u(0,x_N)=|x'|\e(x',x_N)\; ,$$
    where $\e(x',x_N)\to 0$ when $(x',x_N)\to (0,0)$.

    \medskip

    \noindent\textbf{(b)}
    For $r'\leq r$, we consider the domain $D_{r'}:= \{(x',x_N):\ |x'|\leq x_N,\ 0\leq x_N\leq r'\}$ 
    and the function 
    $$\psi(x',x_N):=u(x',x_N)-\frac{|x'|^4}{x_N^4}-p\cdot x_N\;,\quad\text{where}\quad 
    \unp<p<\ovp\;.$$ 
    This function $\psi$ is defined in $D_{r'}\setminus\{(0,0)\}$ but we can
    extend it at $(0,0)$ by setting $\psi(0,0)=0$, which yields an \usc function on $D_{r'}$.

    \medskip

    \noindent\textbf{(c)}
    We consider a sequence $(s_k)_k$ such that $s_k>0$ for all $k$, $s_k\to 0$ and
    $$\frac{u(0,s_k)-u(0,0)}{s_k}=\frac{u(0,s_k)}{s_k}\to \unp\;,$$
    and we choose $r'=s_k$. We claim that the maximum of $\psi$ on $D_{s_k}$ is achieved
    in the interior of the domain.

    Indeed, if $|x'|=x_N$, $\psi(x',x_N)=u(x',x_N)-1-px_N\leq -1+o(1),$
    while for $x_N=r'=s_k$, by the estimate in \textbf{(a)} above,
    \begin{align}
    \psi(x',s_k)\leq & |x'|\e(x',s_k) + u(0,s_k)-ps_k \\
    \leq & (\unp-p)s_k+ o(s_k) <0 \; .
    \end{align}
    Finally $\psi(0,0)=0$ and we conclude that $\psi\leq 0$ on $\partial D_{s_k}$.

    On the other hand, there exists a sequence $(r_l)_l$ such that $r_l>0$ for all $l$, $r_l\to 0$
    and 
    $$\frac{u(0,r_k)-u(0,0)}{r_k}=\frac{u(0,r_k)}{r_k}\to \ovp\; .$$
    For $l$ large enough, we have $r_l<s_k$ and
    \begin{align}
    \psi(0,r_k)\leq & u(0,r_k)-pr_k= (\ovp-p)r_k+ o(r_k) > 0 \; .
    \end{align}
    Therefore the maximum of $\psi$ is achieved in the interior of $D_{s_k}$.

    \medskip

    \noindent \textbf{(d)} We can now apply the viscosity subsolution inequality at the maximum point
    $(x',x_N)$ of $\psi$. To do so, we first remark that $u$ is differentiable \wrt $x'$ at this
    maximum point by the semi-convexity property and
    $$ D_{x'}u(x',x_N)=\frac{4|x'|^2x'}{x_N^4}\;,$$
    therefore $4|x'|^3/x_N^4\to 0$ as $s_k\to 0$\;.
    On the other hand, the term corresponding to the second derivative, taking into account that
    $a(x_N)=O(x_N^2)$ is estimated by 
    $$  -{\rm Tr}(a(x_N)D^2\psi)=
    O(x_N^2)\cdot O \bigg(\frac{|x'|^2}{x_N^4}+\frac{|x'|^3}{x_N^5}+\frac{|x'|^4}{x_N^6}\bigg)\to 0
    \quad \hbox{as  } s_k\to 0\;.$$
    We deduce that
    $$ \tilde H\Big(p-\frac{4|x'|^4}{x_N^5}\Big)\leq o(1)\;,$$
    but since $4|x'|^4/x_N^5\to 0$ as $s_k \to 0$, the conclusion follows: $\tilde H(p)\leq0$.

    \medskip

    \noindent\textbf{(e)} The case when $p=\unp$ or $\ovp$ follows from the continuity of $\tilde
    H$, and the case $\unp=\ovp$ is treated exactly as in the first-order case, so the proof is complete.
\end{proof}

Using Lemma~\ref{HJ-dim1-so}, the proof of Theorem~\ref{comp-IM-so-LS} follows along the lines of
Theorem~\ref{LS-CR-GC}, the Lipschitz continuity of $u$ and $v$ ensuring that the regularization
process allows to use the lemma.

\section{Vanishing viscosity approximation (II): convergence via junction viscosity solutions} 

\index{Vanishing viscosity method!via junction viscosity solutions}

In this section, we use the Lions-Souganidis comparison result to show that the vanishing viscosity
approximation converges to the unique solution of the Kirchhoff problem; this gives another version
of Theorem~\ref{teo:viscous} in a non-convex setting.

\begin{theorem}\label{pro:viscous2}\emph{--- Vanishing viscosity limit via junction viscosity
    solutions.}\smsp
    Assume that, for any $\eps>0$, $u^\eps$ is a continuous viscosity solution of
    \begin{equation}\label{pb:viscous}
       u^\eps_t -\eps \Delta u^\eps + H (x,t,Du^\eps) = 0\quad\text{in}\quad\R^N \times (0,\Tf)\;,
    \end{equation}
    \begin{equation}\label{pb:viscousid}
       u^\eps(x,0) = u_ 0(x) \quad\text{in}\quad\R^N \;,
    \end{equation}
    where $H=H_1$ in $\Omega_1$ and $H_2$ in $\Omega_2$, and $u_0$ is bounded continuous function in
    $\R^N$. Under the assumptions of Theorem~\ref{LS-CR} and if the sequence $(u^\eps)_\eps$ is
    uniformly bounded in $\R^N \times (0,\Tf)$, $C^1$ in $x_N$ in a neighborhood of $\H$, then, as
    $\eps\to0$, the sequence $(u^\eps)_\eps$ converges locally uniformly to the unique solution of
    the Kirchhoff problem in $\R^N \times (0,\Tf)$.
\end{theorem}

This second result on the convergence of the vanishing viscosity approximation may appear as being
more general than Theorem~\ref{teo:viscous} since it covers the case of non-convex Hamiltonians. But
we point out that Theorem~\ref{LS-CR} requires either \eqref{Lip-z-v} or \TCs which limit its range
of applications. This also suggests connections between solutions with \FL and \KC junctions
conditions: we study these connections in the next section.

The proof is almost standard since we use the half-relaxed limits method to pass
to the limit, coupled with a strong comparison result to conclude, here Theorem~\ref{LS-CR}.
So, the only difficulty consists in proving lemma \ref{StabK} below which, despite its very
classical formulation, is not standard at all: the formulation involves test-function which are not
smooth across $\H$. Therefore, this is not an usual stability result for viscosity solutions.

\begin{lemma}\label{StabK}
    The half-relaxed limits $\ou = \limssup u^\eps$ and $\uu = \limiinf u^\eps$ are respectively sub
    and supersolution of the Kirchhoff problem.  
\end{lemma}

\begin{proof}
    We prove the result for $\ou$, the one for $\uu$ being analogous. Let $\phi \in \PC1$ be a
    test-function and let $(\xb,\tb)$ be a strict local maximum point of $\ou-\phi$. The only
    difficulty is when $\xb\in \H$ and therefore we concentrate on this case.

    By standard arguments, $u^\eps-\phi$ has a local maximum point at $(\xe,\te)$ and $(\xe,\te)\to
    (\xb,\tb) $ as $\eps\to 0$. 

    \medskip

    \noindent\textbf{(a)} 
    If there exists a subsequence $(x_{\eps'},t_{\eps'})$ with
    $x_{\eps'}\notin \H$, the classical arguments can be applied and passing to the limit (along
    another subsequence) in the inequality
    $$\phi_t (\xe,\te) -\eps \Delta \phi (\xe,\te) + 
    H(\xe,\te,u^\eps (\xe,\te) ,D\phi(\xe,\te)) \leq 0\;$$
    yields the result.

    \medskip

    \noindent\textbf{(b)}
    The main difficulty is when $x_{\eps}\in \H$ for all $\eps$ small enough since $\phi$ is not
    smooth at $(\xe,\te)$. Here we use 
    $$a:=\frac{\partial \phi}{\partial x_N}((\xb',0+),\tb) = 
    \lim_{\displaystyle {\substack{x_N\to 0\\ x_N>0}}}\, \frac{\partial \phi}{\partial x_N}((\xb',x_N),\tb)$$
    and
    $$b:= \frac{\partial \phi}{\partial x_N}((\xb',0-),\tb) = 
    \lim_{\displaystyle {\substack{  x_N\to 0\\ x_N<0}}}\, \frac{\partial \phi}{\partial x_N}((\xb',x_N),\tb)\; .$$
    If $-a+b\leq 0$,
    the Kirchhoff subsolution condition is satisfied and the result holds.

    \medskip

    \noindent\textbf{(c)}
    The last possibility is that $-a+b>0$, and since $\ue$ is smooth, the maximum point property at
    $(\xe,\te)$ implies that 
    $$ \frac{\partial \ue}{\partial x_N}((\xe',0),\tb) \leq 
    \frac{\partial \phi}{\partial x_N}((\xe',0+),\tb)\; ,$$
    and
    $$ \frac{\partial \ue}{\partial x_N}((\xe',0),\tb) \geq 
    \frac{\partial \phi}{\partial x_N}((\xe',0-),\tb)\; .$$
    Therefore 
    $$ -\frac{\partial \phi}{\partial x_N}((\xe',0+),\tb)+ 
    \frac{\partial \phi}{\partial x_N}((\xe',0-),\tb) \leq 0\; ,$$ 
    and using that both partial derivatives are continuous in $x'$, we reach a contradiction for
    $\eps$ small enough (remember that $-a+b>0$ here).
\end{proof}

\chapter{From One Notion of Solution to the Others}
\label{sect:equiv.sols}

\abstract{This chapter focuses on the connections between Ishii viscosity solutions, flux-limited
solutions and junction viscosity solutions, in the case of quasi-convex Hamiltonians for which all
these notions of solutions make sense. Various results are given, among which the most surprising is
that a junction viscosity solutions associated to a problem with a Kirchhoff type junction condition
can be seen as a flux-limited solutions for some well-chosen flux~limiter. As an illustration, the
problem of the convergence of the vanishing viscosity method is addressed. Remarks on the existence
of solutions are also given.}

The aim of this chapter is to connect the three notions of solutions we have at hand: Ishii
solutions, flux-limited solutions and junction viscosity solutions. 

We first show that Ishii solutions for \eqref{pb:half-space} can be seen as \FLS associated with an
$\HT$ or $\HTreg$ flux~limiter in the case of quasi-convex Hamiltonians.  We point out
that the definitions of $\HT$ and $\HTreg$ are extended to the case of quasi-convex Hamiltonians by
\eqref{HT1} and \eqref{HT2}, and we refer the reader to Section~\ref{sect:quasi.convexity} for useful
results on them. Through this \FLS interpretation, we complement the results of
Part~\ref{part:codim1} both by taking into account more general Hamiltonians but also by
considering a notion of solution which allows a pure pde approach of the problem.

Then we compare \FLS and \JVS in the context of flux-limited conditions: here the formulation of the
notion of solutions on $ \H$ is the key point. 

Finally, we prove that junction viscosity solutions associated with Kirchhoff conditions are
flux-limited solutions for a specific flux~limiter which we identify explicitly. We do the analysis
first for the most classical Kirchhoff condition and then for generalized ones.

\section{Ishii and flux-limited solutions}

The main result of this section is the
\begin{proposition}\label{prop:cvs-fls}
    Assume that \HQC holds. Then
    \begin{enumerate}
        \item[$(i)$] An \usc function $u:\R^N\times (0,\Tf) \to \R$ is an Ishii subsolution of
        \eqref{pb:half-space} if and only if it is a \FLSub of \HJgen-\FL with the flux~limiter
        $\HTreg$. 
        \item[$(ii)$] A \lsc function $v:\R^N\times (0,\Tf) \to \R$ is an Ishii supersolution of
        \eqref{pb:half-space} if and only if it is a \FLSup of \HJgen-\FL with the flux~limiter
        $\HT$.
    \end{enumerate}
\end{proposition}

An immediate corollary of this result is the
\begin{corollary}\label{cor:UpUmgen}Under the assumptions of Theorem~\ref{comp-IM-nc} with $G=\HT$
    or $\HTreg$,
    \begin{enumerate}
        \item[$(i)$] If $\Up$ is the unique \FLS of \HJgen-\FL with the flux~limiter $\HTreg$, it is the maximal Ishii subsolution of
        \eqref{pb:half-space}.
    \item[$(ii)$] If $\Um$ is the \FLS of \HJgen-\FL with the flux~limiter $\HT$, it is a minimal Ishii supersolution of
        \eqref{pb:half-space}.
    \end{enumerate}
\end{corollary}

\begin{proof}[Proof of Proposition~\ref{prop:cvs-fls}]
    Of course, only the viscosity inequalities on $\H \times (0,\Tf)$ are different and therefore we
    concentrate on them.

    \medskip

    \noindent\textbf{(a)} \emph{Generalities ---}
    Throughout the proof we consider elements of the superdifferential of $u$ or the subdifferential
    of $v$ at $(x,t) \in \H \times (0,\Tf)$ of the form $(p_x+\lambda e_N, p_t)$, and inequalities
    for $H_1, H_2$ and $\HT$ or $\HTreg$. In all these inequalities, only the dependence in
    $\lambda$ is going to play a role and, in order to simplify the notations, we set, for $i=1,2$
    $$ 
         F_i (\lambda):=  p_t +H_i (x,t,r,p_x+\lambda e_N)\;,
    $$
    where $x$, $t$, $p_x$ and $r=u(x,t)$ or $v(x,t)$ are assumed to be fixed. We define analogously
    $ F_i^{\pm}$. We also use the notations
    $$\FT=p_t + \HT(x,t,r,p_x)\;,\; \FTreg=p_t +  \HTreg(x,t,r,p_x)\;,$$ 
    for, again, $x$, $t$, $r=u(x,t)$ or $v(x,t)$ and $p_x$ being fixed. We also recall that, in the
    quasi-convex case, 
    $$ \HT(x,t,r,p_x)=\min_{\lambda\in\R} \max(H_1 (x,t,r,p_x+\lambda e_N),H_2 (x,t,r,p_x+\lambda e_N))\;,$$
    $$ \HTreg(x,t,r,p_x)=\min_{\lambda\in\R} \max(H_1^- (x,t,r,p_x+\lambda e_N),H_2^+ 
    (x,t,r,p_x+\lambda e_N))\;,$$
    \ie $\FT= \min_\lambda \max(F_1(\lambda), F_2(\lambda))$ and similarly
    $\FTreg= \min_\lambda \max(F_1^-(\lambda),  F_2^+(\lambda))$. From these
    representations we deduce easily from Section~\ref{sec:max-qc} the existence of
    $(\underline\lambda,\overline\lambda)\in\R^2$ such that $\FTreg=
    F_1^-(\overline\lambda)= F_2^+(\overline\lambda)$ and 
    \begin{equation}\label{eq.cases.HT}
        \FT=\begin{cases}
        & \text{either }  F_1(\underline\lambda)= F_2(\underline\lambda)\\
        & \text{or } F_1(\underline\lambda)=\min  F_1(\lambda) \text{ if }
         F_1\geq F_2\\
        & \text{or } F_2(\underline\lambda)=\min  F_2(\lambda)\text{ if }
         F_2\geq F_1\;.
    \end{cases}
    \end{equation}

    \medskip

    \noindent\textbf{(b)} \emph{Subsolution case ---} 
    Before providing the proof, we point out that \HQC implies that \NCw holds for all
    the Hamiltonians involved in the Ishii and flux-limited formulations, \cf
    Remark~\ref{rem:reg-sub-weak}. Hence all the subsolutions we are going to consider are regular
    on $\H$.

    If $u$ is an Ishii subsolution of \eqref{pb:half-space} and if we pick $(p_x+\lambda_1 e_N,
    p_t)\in D_{\overline{Q_1}}^+ u (x,t)$, $(p_x+\lambda_2 e_N, p_t)\in D_{\overline{Q_2}}^+ u (x,t)$, we have
    to show that 
    $$\max( F_1^+(\lambda_1), F_2^-(\lambda_2),\FTreg) \leq 0\; .$$
    Notice that Ishii subsolutions are Lipschitz continuous. Therefore, not only are they regular on
    $\H$ but the normal components of their superdifferentials are bounded from below in
    $\overline{Q_1}$ and from above in $\overline{Q_2}$. Proposition~\ref{sub-ineq-on-b} can then be
    applied, which gives existence of  $\mu_1 \leq \lambda_1$ and $\mu_2 \geq \lambda_2$ such that
    $(p_x+\mu_1 e_N, p_t)\in D_{\overline{Q_1}}^+ u (x,t)$, $(p_x+\mu_2 e_N, p_t)\in D_{\overline{Q_2}}^+ u
    (x,t)$ and
    $$F_1(\mu_1)\leq 0 \; , \; F_2(\mu_2)\leq 0\; .$$
    Since these inequalities imply $ F_1^+ (\mu_1)\leq 0 $ and $ F_2^-(\mu_2)\leq 0$,
    we deduce from the monotonicity of $F_1^+$ and $F_2^-$ that $
    F_1^+(\lambda_1)\leq 0$ and $F_2^-(\lambda_2)\leq 0$.

    It remains to prove that $\FTreg \leq 0$. To do so, we argue by contradiction: if $\FTreg >0$,
    then $F_1^-(\overline\lambda)= F_2^+(\overline\lambda)>0$. Since $
    F_1^-(\mu_1)\leq 0$ and $F_2^+(\mu_2)\leq 0$, we deduce from the monotonicity of $
    F_1^-$ and $F_2^+$ that $\mu_1<\overline\lambda<\mu_2$. But this double inequality
    implies that $$(p_x+\overline\lambda e_N, p_t)\in D_{\R^N}^+ u (x,t)$$ and therefore
    $\min(F_1(\overline\lambda),F_2(\overline\lambda))\leq 0$ because $u$ is an Ishii
    subsolution of \eqref{pb:half-space}. This is a contradiction with the fact that $
    F_1^-(\overline\lambda)=F_2^+(\overline\lambda)>0$, and we deduce that the
    \FLSub condition hold son $\H$.

    \medskip 

    Conversely, if $u$ is \FLSub of \HJgen-\FL associated with the flux~limiter $\HTreg$ and if
    $(p_x+\lambda e_N, p_t)\in D_{\R^N}^+ u (x,t)$, we want to show that 
    $$ \min( F_1(\lambda), F_2(\lambda))\leq 0\; .$$
    By the \FLSub property, we know that 
    $$\max(F_1^+(\lambda), F_2^-(\lambda),\FTreg) \leq 0\; ,$$
    and therefore it remains to prove that $\min(F_1^-(\lambda), F_2^+(\lambda))\leq
    0$. We argue by contradiction assuming that this min is strictly positive. Using again the
    monotonicity of $ F_1^-,F_2^+$ and the fact that $\FTreg \leq 0$, we deduce that 
    $\lambda > \overline\lambda$ and $\overline\lambda > \lambda$ which is clearly a contradiction, 
    proving that $u$ is an Ishii subsolution. The proof for the subsolution case is then complete.
    
    \medskip

    \noindent\textbf{(c)} \emph{Supersolution case ---} 
    Contrarily to the case of subsolutions, the supersolutions are not necessarily regular on $\H$
    but as we see below, this does not pose any problem.

    If $v$ is an Ishii supersolution of \eqref{pb:half-space} and if
    $(p_x+\lambda_1 e_N, p_t)\in D_{\overline{Q_1}}^- v (x,t)$, 
    $(p_x+\lambda_2 e_N, p_t)\in D_{\overline{Q_2}}^- v (x,t)$, we want to show that
    $$\max( F_1^+(\lambda_1),F_2^-(\lambda_2),\FT) \geq 0\; .$$
    In order to apply Proposition~\ref{sub-ineq-on-b}, we consider several cases
    \begin{enumerate}
        \item[1.] If the set $J_1= \{\lambda \in \R: (p_x+\lambda
            Dd(x),p_t) \in D_{\overline{Q_1}^l}^-v (x,t)\}$ is bounded from above,
            there exists $\mu_1 \geq \lambda_1$ such that 
            $(p_x+\mu_1 e_N, p_t)\in D_{\overline{Q_1}}^- v (x,t)$ and $F_1(\mu_1)\geq 0$.
    \item[2.] If the set $J_2= \{\lambda \in \R: (p_x+\lambda
            Dd(x),p_t) \in D_{\overline{Q_2}^l}^-v (x,t)\}$ is bounded from below,
            there exists $\mu_2 \leq \lambda_2$ such that $(p_x+\mu_2 e_N, p_t)\in
            D_{\overline{Q_2}}^- v (x,t)$ and $ F_2(\mu_2)\geq 0$.  \item[3.] Otherwise, for any $\mu_i
            \in \R$, $(p_x+\mu_i e_N, p_t)\in  D_{\Omegb_i}^- v (x,t)$ and therefore, by the
            coercivity of $F_i$, $F_i(\mu_i)\geq 0$ for some $\mu_i$ such that $\mu_i \geq
            \lambda_1$ if $i=1$ or
            $\mu_i \leq \lambda_2$ if $i=2$.
    \end{enumerate}

    In any case, if either $F_1^+ (\mu_1)\geq 0 $ or $F_2^-(\mu_2)\geq 0$, we are
    done by using the monotonicity of $F_1^+$ and $F_2^-$. Therefore we can assume
    without loss of generality that $F_1^- (\mu_1), F_2^+(\mu_2)\geq 0$
    and we have to prove that $\FT\geq 0$.

    We argue by contradiction assuming that $\FT<0$ and using \eqref{eq.cases.HT}, we see that
    there are three options.

    \begin{enumerate}
        \item[$(i)$] If there exists $\underline \lambda$ such that $\FT=
    F_1(\underline\lambda)=F_2(\underline\lambda)<0$, then by the monotonicity of $
    F_1^-$ and $ F_2^+$, it follows that $\mu_1 >\underline\lambda >\mu_2$. Hence
    $(p_x+\underline\lambda e_N, p_t)$ is both in $D_{\overline{Q_1}^l}^-v (x,t)$ and in $D_{\overline{Q_2}^l}^-v (x,t)$,
    and therefore in $D_{\R^N}^- v (x,t)$ and by the Ishii supersolution
    property, we get $\max(F_1(\underline\lambda),F_2(\underline\lambda))\geq 0$, a
    contradiction. 

    \item[$(ii)$] If $\FT=F_1(\underline\lambda)=\min
    F_1(\lambda)<0$ with $F_1\geq  F_2$, we still have $
    F_2(\underline\lambda)<0$ and we conclude in the same way. 

    \item[$(iii)$] Of course this is also the case when
    $\FT=F_2(\underline\lambda)=\min F_2(\lambda)<0$ with  $F_2\geq 
    F_1$.
    \end{enumerate}

    In conclusion, $\FT\geq0$ and $v$ is \FLSup associated to the flux~limiter $\HT$. 

    \medskip

    Conversely, if $v$ is \FLSup of \HJgen-\FL with the flux~limiter $\HT$ and if $(p_x+\lambda e_N,
    p_t)\in D_{\R^N}^- v (x,t)$, we have to show that 
    $$ \max(F_1(\lambda), F_2(\lambda))\geq 0\; .$$
    By the \FLSup property, we already know that 
    $$\max(F_1^+(\lambda), F_2^-(\lambda),\FT) \geq 0\; ,$$
    which implies that: $(i)$ either $F_1(\lambda)\geq F_1^+(\lambda)\geq0$ or $F_2(\lambda)\geq
    F_2^-(\lambda)\geq0$, in which case we are done; $(ii)$ or $F_T\geq0$, but using that
    $\max(F_1^-(\lambda), F_2^+(\lambda)) \geq \FT \geq 0$ implies that
    $\max(F_1(\lambda),F_2(\lambda))\geq0$ and we also get the conclusion.
    The proof for the supersolution case is then complete.
\end{proof}

\section{Flux-limited and junction viscosity solutions for flux-limited conditions}
\index{Flux-Limited condition}

We now prove the equivalence of both notions of solutions in the case of Flux-Limited conditions. We
point out that, since \JVSub are not necessarily regular, we have to make this non-trivial
assumption. However, as we saw in Proposition~\ref{prop:JVreg}, this assumption is automatically
satisfied in the case of Kirchhoff-type conditions.
\begin{proposition}\label{prop:equivFL-LS} 
    Assume \GAQC and that $G$ satisfies \HBAHJ. Then
    \begin{enumerate}
    \item[$(i)$] an \usc, locally bounded function $u: \R^N \times (0,\Tf) \rightarrow \R$ is a
    flux-limited subsolution of \HJgen-\FL with flux~limiter $G$ if and only if it is a
    regular $G$-\JVSub.
    \item[$(ii)$] a \lsc, locally bounded function $v: \R^N \times (0,\Tf) \rightarrow \R$ is a
    flux-limited supersolution of \HJgen-\FL if it is a $G$-\JVSup.
    \end{enumerate}
\end{proposition}

\begin{proof}
    In all this proof, $\psi$ is always a generic test-function in $\PC1$ and the maximum or minimum of
    $u-\psi$ in $\R^N\times(0,\Tf)$ is always denoted by $(x,t)$, which we assume to be located on
    $\H\times(0,\Tf)$.

    \noindent\textbf{(a)} \emph{Subsolutions --}
    We just sketch the proof here since this case is easy.  If $u$ is a flux-limited subsolution, it
    clearly satisfies \eqref{cvsub-n}.  Indeed, if $u-\psi$ has a maximum at $(x,t)$, then 
    $\psi_t+G(x,t,v,D_\H \psi) \leq 0$ because of the ``$\max$'' in the definition of flux-limited
    subsolutions.  To prove the
    converse, we use in an essential way the regularity of the \JVSub: using
    Proposition~\ref{sub-up-to-b} with $L=H_1^+$ or $H_2^-$, we see that $(x,t) \in
    \H\times (0,\Tf)$, local maximum point of $u-\psi$ in $\R^N\times (0,\Tf)$ then 
    $$\psi_t+ H_1^+(x,t, u, D\psi_1) \leq 0\quad , \quad  \psi_t+ H_2^-(x,t, u, D\psi_2) 
    \leq 0\; .$$
    It remains to prove that $\psi_t+ G(x,t,v,D_\H \psi) \leq 0$, which is done as follows: for any
    $C>0$, $u-(\psi+C|x_N|)$ has also a maximum at $(x,t)\in\H\times(0,\Tf)$ but taking $C>0$ large
    enough in \eqref{cvsub-n} yields that the min cannot be reached by the $H_1$/$H_2$-terms since
    both Hamiltonians are coercive. Thus necessarily, the non-positive min is given by the junction
    condition and the result follows.

    \medskip

    \noindent\textbf{(b)} \emph{Supersolutions --}
    This case is a little bit more delicate. Of course, a flux-limited supersolution $v$ satisfies
    \eqref{cvsuper-n} since $H_1\geq H_1^+$ and $H_2\geq H_2^-$. The main point is then to prove that
    supersolutions of \eqref{cvsuper-n} are flux-limited supersolutions.

    If $(x,t) \in \H\times (0,\Tf) $ is a local maximum point of $u-\psi$, \eqref{cvsuper-n} holds
    and we wish to show that 
    \[
      \max \Big(\psi_t+ G(x,t,v,D_\H \psi), \psi_t+ H^+_1(x,t, v, D\psi_1) ,  
      \psi_t+ H^-_2(x,t, v, D\psi_2) \Big) \geq 0  \: .
    \]
    Assuming this is not the case, then necessarily all three quantities above are strictly negative
    and \eqref{cvsuper-n} implies
    $$  \max \Big(\psi_t+ H^-_1(x,t, v, D\psi_1), \psi_t+ H^+_2(x,t, v, D\psi_2) \Big) \geq 0\; .$$
    Let us assume for example that $\psi_t+ H^-_1(x,t, v, D\psi_1)\geq 0$, the other case being treated
    similarly.

    Referring the reader to Section~\ref{upoH} and Remark~\ref{rem:upoH} where the properties of
    $H^+_1, H_2^-$ are described we see that, if $D\psi_i=p^i_T+ p^i_N e_N$ for $i=1,2$, where
    $p^i_T \in \H$ and $p^i_N\in \R$, then these inequalities imply for instance
    $$
    -\psi_t(x,t) > H_1^+(x,t,v,p^1_T+p^1_Ne_N)\;,\;\hbox{therefore}\; -\psi_t(x,t)>\min_s
    (H_1(x,t,v,p^1_T+s e_N))\;.$$ 
    Denoting by $s_*\in\R$ a real such that $-\psi_t(x,t)=H_1^-(x,t,v,p ^1_T+se_N)$, we deduce that
    $s_*>m_1^+(x,t,v,p^1_T)$, the largest point of where $s\mapsto  H_1(x,t,v,p^1_T+s e_N)$ reaches its minimum. 
    On the other hand, the inequality $\psi_t+ H^-_1(x,t, v, D\psi_1)\geq 0$ implies that
    $ p^1_N \geq s_*$, so that finally $p^1_N > m_1^+(x,t, v, p^1_T)$.

    There are now two cases. In the first case $\psi_t+ H^+_2(x,t, v, D\psi_2)\geq 0$ and
    similarly as above, $p^2_N < m_2^-(x,t, v, p^1_T)$, the least minimum point for $H_2$. Here, we
    set 
    $$\tilde\psi(x,t):=\begin{cases} \tilde\psi_1(x,t)=\psi_1(x,t)+m_1^+(x,t, v,
    p^1_T)-p^1_N)x_N & \text{if }x_N>0\\
    \tilde\psi_2(x,t)=\psi_2(x,t)+m_2^-(x,t, v, p^2_T)-p^2_N)x_N & \text{if }x_N<0\;.
    \end{cases}$$
    This new test-function still belongs to $\PC1$ and $v-\tilde \psi$ has still a minimum point at
    $(x,t)$, therefore \eqref{cvsuper-n} holds with $\tilde\psi$. But, since by construction
    $D\tilde\psi_1(x,t)=m_1^+(x,t, v, p^1_T)$ while $D\tilde\psi_2(x,t)=m_2^- (x,t, v, p^2_T)$, it
    follows that for $i=1,2$, $$ \tilde \psi_t+ H_i(x,t, u, D\tilde \psi_i) = \tilde  \psi_t+ \min_s
    (H_i(x,t,v,p^i_T+s e_N)) < 0\; .$$
    Therefore $\tilde \psi_t+ G(x,t,v,D_\H \tilde  \psi) \geq 0$, which obviously implies 
    $\psi_t+G(x,t,v,D_\H \psi)\geq 0$, so that the flux-limited condition holds.

    If, on the contrary, $\psi_t+ H^+_2(x,t, v, D\psi_2)< 0$, then $\psi_t+ H_2(x,t,v,D\psi_2)<0$ 
    and the change of test-function reduces to 
    $$\tilde\psi(x,t):=\begin{cases}
    \tilde\psi_1(x,t)=\psi_1(x,t)+m_1^+(x,t, v, p^1_T)-p^1_N)x_N & \text{if }x_N>0\\
    \tilde\psi_2(x,t)=\psi_2(x,t) & \text{if }x_N<0\;,
    \end{cases}$$
    but we conclude as in the first case, which ends the proof. 
\end{proof}

\section{The Kirchhoff condition and flux~limiters}
\label{sec:KCvsFL}

\index{Kirchhoff condition}
Here we compare the sub/supersolution of \HJgen associated with the Kirchhoff condition \KC on one
hand, and \FL-conditions on the other hand in the framework of quasi-convex Hamiltonians. We
also consider the cases of more general Kirchhoff type conditions. To simplify matter, we also drop
here the dependence of the Hamiltonians in $u$ since this does not create much more difficulty in
the proofs.

The results of this section are based on the analysis of various properties of the Hamiltonians (in
particular $\HTreg$) which first appear in Section~\ref{upoH}, taking into account
Remark~\ref{rem:upoH}.  We again recall that the definitions of $\HT$ and $\HTreg$ are extended to
the case of quasi-convex Hamiltonians by \eqref{HT1} and \eqref{HT2} and we refer the reader to
Section~\ref{sect:quasi.convexity} for useful results on them. Notice that these sections are
written in a slightly more general form, where the Hamiltonians depend on $u$ for the sake of
completeness but the results apply here, of course.

Our main result is
\begin{proposition}\label{kc-fl}Assume \GAQC. 
    \begin{enumerate}
    \item[$(i)$] An \usc function $u$ is a \JVSub of \HJgen-\KC 
        if and only if $u$ is a \FLSub with $G=\HTreg$.
    \item[$(ii)$] A \lsc function $v$ is a \JVSup of \HJgen-\KC 
        if and only if $v$ is a \FLSup with~$G=\HTreg$.
    \end{enumerate}
\end{proposition}

It is worth pointing out that this result holds both in the convex and non-convex case, provided
that \HQC is satisfied.

\begin{proof}
    Of course, in both results, only the viscosity inequalities on $\H$ are different and therefore
    we concentrate on this case. Again we are going to use the results of Section~\ref{upoH} in
    light of Remark~\ref{rem:upoH}.
    
    \medskip

    \noindent\textbf{(a)}
    We begin with the simpler implication that a \FLSub \resp{\FLSup} with $G=\HTreg$ is
    a \JVSub \resp{\JVSup} of \HJgen-\KC. This is a consequence of the properties
    \begin{equation}\label{equiv.def.minmax}
        \begin{aligned}
        \HTreg(x,t,p')= &\min_{s\in\R} \max \big(H_{1}^- (x,t,p^\prime+s e_N),H_{2}^+ 
        (x,t,p^\prime+s e_N)\big)\\
        = &\max_{s\in\R} \min\big(H_{1}^- (x,t,p^\prime+s e_N),H_{2}^+ (x,t,p^\prime+s e_N)\big)\; ,
        \end{aligned}
    \end{equation}
    the first equality being the definition of $\HTreg$, the second one being and easy consequence
    of the monotonicity property of $H_{1}^- ,H_{2}^+$.

    We just sketch the proof dropping the variables $x,t,p^\prime$ and keeping only the one
    corresponding to the $x_N$-derivative for the sake of clarity and denote by $a$ the
    $u_t$-variable.

    For the \FLSub case, we start from
    \begin{equation}\label{equiv.JVSub} 
        \max(a+H_1^+(p_1), a+H_2^-(p_2),a+\HTreg)\leq 0
    \end{equation}
    and using \eqref{equiv.def.minmax} in the $\max\min$ form, we get both
    $$\HTreg \geq \min\big(H_{1}^- (p_1),H_{2}^+ (p_1)\big)\text{ and } 
    \min\big(H_{1}^- (p_2),H_{2}^+ (p_2)\big) \; .$$
    Now, there are two possibilities:
    \begin{enumerate}
        \item[1.] Either $-p_1+p_2 \leq 0$, in which case we clearly get the \KC condition:\\
        $\max(a+H_1(p_1), a+H_2(p_2),-p_1+p_2)\leq 0$.
        
        \item[2.] Or $p_1>p_2$ and the monotonicity of $H_{2}^+$ implies 
            $\HTreg \geq \min\big(H_{1}^- (p_1),H_{2}^+ (p_2)\big)\;,$ which leads to
            $$\max\big(a+H_1^-(p_1),a+H_2^+(p_2)\big)\leq0\;.$$
            But combining this inequality with \eqref{equiv.JVSub}, we also end up with the \KC condition: 
            $$\min(a+H_1(p_1), a+H_2(p_2),-p_1+p_2)\leq 0\;.$$
    \end{enumerate}

    The \FLSup case is done by similar arguments, using this time \eqref{equiv.def.minmax} in the
    $\min\max$ form.  Now we turn to the proofs that \JVSub \resp{\JVSup} of \HJgen-\KC are \FLSub
    \resp{\FLSup} with $G=\HTreg$.

    \medskip

    \noindent\textbf{(b)} \emph{Subsolutions --} 
    We first recall that, by Proposition~\ref{prop:JVreg}, the \JVSub of \HJgen-\KC are regular on
    $\H$. As a consequence, the $H_1^+, H_2^-$ inequalities clearly hold on $\H$ thanks to
    Proposition~\ref{sub-up-to-b} with $L=H_1^+$ or $H_2^-$.
    
    Hence we 
    just have to prove that, if $(\xb,\tb)\in \H \times (0,\Tf)$ is a strict local maximum point of
    $u-\psi$ for some function $\psi =(\psi_1,\psi_2) \in \PC1$, then
    \begin{equation}\label{ineq:sub.FL.KC}
        \psi_t (\xb,\tb)+\HTreg\big(\xb,\tb,D_{\H}\psi (\xb,\tb)\big)\leq 0\;.
    \end{equation}
    In particular, $(\xb,\tb)$ is a strict local maximum point of $((x',0),t) \mapsto
    u((x',0),t)-\psi((x',0),t)$ on $\H$. Now, in order to build a specific test-function, 
    we consider for some small $\kappa >0$
    $$ \chi(y_N):= \begin{cases}
     (\lambda -\kappa)y_N & \mbox{ if }  y_N \geq 0\;,\\
     (\lambda +\kappa)y_N & \mbox{ if }  y_N < 0\;,
    \end{cases}
    $$
    where, referring to Lemma~\ref{lem:H1m.H2p.a}, $\lambda$ is a minimum point of the coercive,
    continuous function $s\mapsto \max \big(H_{1}^- (\xb,\tb,D_{x'}\psi(\xb,\tb)+s e_N),H_{2}^+
    (\xb,\tb,D_{x'}\psi(\xb,\tb)+s e_N)\big)$. Notice that by this
    lemma,
     \begin{equation}\label{condlamb}
     \HTreg(\xb,\tb,D_{x'}\psi(\xb,\tb))=H^-_1(\xb,\tb,D_{x'}\psi(\xb,\tb)+\lambda e_N)=
      H^+_2(\xb,\tb,D_{x'}\psi(\xb,\tb)+\lambda e_N)\;,
     \end{equation}
    By standard arguments, the following function
    \begin{equation}\label{test}
        (x,t) \mapsto u(x,t)-\psi((x',0),t)-\chi(x_N)-\frac{(x_N)^2}{\eps^2}
    \end{equation}
     has a  maximum point $(\xe,\te)$ near~$(\xb,\tb)$ and $(\xe,\te)\to (\xb,\tb)$ as $\e$ tends to
     $0$ since $(\xb,\tb)$ is a strict local maximum point of $(x,t) \mapsto u(x,t)-\psi((x',0),t)$
     on $\H$.
    
    Notice that since $\kappa>0$, choosing $\chi$ as above prevents the \KC-condition to hold on
    $\H$, hence the condition on $\H$ reduces to ``$\min\big(\psi_t+H_1,\psi_t+H_2\big)\leq0$''.
    Now we examine the quantity 
    $$Q_{\e} := H_1\Big(\xe,\te,D_{x'}\psi((\xe',0),\te)+(\lambda-\kappa) e_N +
    \frac{2(\xe)_N}{\eps^2} )\Big)\;,$$
    defined only if $(\xe)_N \geq 0$. Since $H_1 \geq H_1^-$ and $H_1^-$ is increasing in the
    $e_N$-direction, it follows that 
    $$\begin{aligned}
    Q_{\e} \geq & \ H_1^-\Big(\xe,\te,D_{x'}\psi((\xe',0),\te)+(\lambda-\kappa) e_N +
        \frac{2(\xe)_N}{\eps^2} )\Big)\\
        \geq &  \ H_1^-\big(\xe,\te,D_{x'}\psi((\xe',0),\te)+(\lambda-\kappa) e_N )\big)\\[2mm]
        \geq &  \ H_1^-\big(\xb,\tb,D_{x'}\psi(\xb,\tb)+(\lambda-\kappa) e_N )+o_\e(1)\\[2mm]
    = & \ \HTreg(\xb,\tb,D_{x'}\psi(\xb,\tb)+\lambda e_N)+o_\e (1) + O(\kappa)\; .
    \end{aligned}
    $$
    An analogous inequality holds if $(\xe)_N \leq 0$ with $H_2$ and $H_2^+$ and we deduce 
    \eqref{ineq:sub.FL.KC} necessarily holds on $\H$.

    \medskip

    \noindent\textbf{(c)} \emph{Supersolutions --}
    Let $v$ be a \JVSup of \HJgen-\KC: we have to prove that $v$ is a flux-limited supersolution
    with~$G=\HTreg$.
    
    To do so, we consider a test-function $\psi = (\psi_1,\psi_2) \in\PC1$ such that $v-\psi$
    reaches a local strict minimum at $(\xb,\tb) \in \H\times (0,\Tf)$. For $i=1,2$, we use the
    notations 
    $$ a = \psi_t (\xb,\tb)\; ,\; p'=D_{x'}\psi(\xb,\tb) \; ,\; \lambda_i = \frac{\partial
    \psi_i}{\partial x_N}(\xb,\tb)\; .$$
    By the supersolution property of $v$, dropping the dependence in $\xb,\tb, p'$ 
    to simplify the notations,
    $$\max\Big( -\lambda_1 + \lambda_2 ,a+H_1(\lambda_1), a+H_2(\lambda_2)\,\Big)\geq 0\;,$$
     and we want to prove that
     $$ \max\Big( a +\HTreg,a+H_1^+(\lambda_1), a+H_2^-(\lambda_2)\Big)\geq 0\; .$$
    We argue by contradiction assuming that in the latter inequality, each term is strictly negative.
    
    With the notations of Section~\ref{sect:eqn-on-b}, we look at the subdifferential of $v$ at
    $(\xb,\tb)$, restricted to each domain $\overline{Q_i}^\ell:=\Omegb_i\times (0,\Tf)$
    for $i=1,2$, and see that $((p',\lambda_i),a) \in D^-_{\overline{Q_i}^\ell}v(\xb,\tb)$. Now
    we apply Proposition~\ref{sub-ineq-on-b}, denoting by 
    $$\begin{aligned}\tilde \lambda_1:= &\sup\big\{\lambda\in\R: ((p',\lambda),a) \in
    D^-_{\overline{Q_1}^l}v(\xb,\tb)\big\}\;,\\ 
        \tilde \lambda_2:= & \inf\big\{\lambda\in\R: ((p',\lambda),a) \in
    D^-_{\overline{Q_2}^l}v(\xb,\tb)\big\}\;,
    \end{aligned}$$ 
    and we point out that $Dd(x)=e_N$ on $\Omega_1$ while $Dd(x)=-e_N$ on $\Omega_2$, which explains the
    difference supremum-infimum. We assume that both quantities are finite and explain at the end of
    the proof that the other cases can be treated by similar and simpler arguments. 
    
    The fact that both $\tilde \lambda_1,\tilde \lambda_2$ are finite implies that $v$ is regular at
    $(\bar x,\bar t)$ and Proposition~\ref{sub-ineq-on-b} implies 
    \begin{equation}\label{Hi-inequa}
        a+H_1(\tilde \lambda_1)\geq 0\quad\hbox{and}\quad  a+H_2(\tilde \lambda_2)\geq 0\; .
    \end{equation}
     Recall that $s\mapsto H_1^+(s)$ is nonincreasing and we are assuming $a+H_1^+(\lambda_1)<0$.
     Therefore, $\tilde \lambda_1 \geq \lambda_1$ implies that also $a+H_1^+(\tilde\lambda_1)<0$.
     In the same way, $a+H_2^-(\tilde \lambda_2)< 0$ which both imply that
    \begin{equation}\label{eq:tildelambda12}
        a+H_1^-(\tilde \lambda_1)\geq 0 \quad\hbox{and} \quad a+H_2^+(\tilde \lambda_2)\geq 0\;.
    \end{equation}

    Taking into account the definition of $\nu_1,\nu_2$ in Lemma~\ref{lem:H1m.H2p.a} and 
    the fact that we assume $a+\HT^\reg(p')<0$, the monotonicity properties of $H_1^-$ and $H_2^+$
    imply that $\tilde \lambda_2 < \nu_1 \leq \nu_2 < \tilde \lambda_1$. 
    Moreover, since $a+H_1^-(\nu_2)=a+\HT^\reg(p')<0$ and $a+H_1^-(\tilde \lambda_1)\geq 0$, there
    exists $\delta_2 \in (\nu_2, \tilde \lambda_1)$ such that 
    $$a+H_1^-(\delta_2)=\frac12\Big(a+\HT^\reg(p')\Big)\;.$$ 
    Since $H_1^-(\delta_2 )>H_1^-(\nu_2 )$, it follows that
    $\delta_2>m_1^-(x,t,p')$, in other words $\delta_2$ belongs to the region where $s\mapsto
    H_1^-(x,t,p'+se_N)$ is increasing, and as a consequence,
    $a+H_1^-(\delta_2)=a+H_1(\delta_2 )$ 

    Similarly, there exists $\delta_1 \in (\tilde \lambda_2 ,\nu_1)$ such that 
    $ H_2^+(\tilde \delta_1 )=H_2(\tilde \delta_1 )=(a+\HT^\reg(p'))/2$ and by
    Proposition~\ref{sub-ineq-on-b} on the structure of the sub-differential, we see that 
    $$
    ((p',\delta_2),a)\in D^-_{\overline{Q_1}^l}v(\xb,\tb)\;,\qquad
    ((p',\delta_1),a)\in D^-_{\overline{Q_2}^l}v(\xb,\tb)\;,
    $$
    which leads to
    $$\max\big(-\delta_2 + \delta_1  ,a+H_1(p'+ \delta_2), a+H_2(p'+\delta_1)\big)\geq 0\;.$$
    But we reach a contradiction here: clearly $-\delta_2+\delta_1<0$, and the other terms are
    obviously negative by the construction of $\delta_1, \delta_2$.

    We finally remark that the key property we use in the proof is \eqref{eq:tildelambda12}, \ie
    roughly speaking, the existence of $\tilde \lambda_1, \tilde \lambda_2$ in the subdifferential
    for which such inequalities hold. If $v$ is not regular on one side (either on $\overline{Q_1}$
    or $\overline{Q_2}$), then any $\lambda \in \R$ is in the corresponding subdifferential and
    therefore  \eqref{eq:tildelambda12} is a consequence of the coercivity of either $H_1^-$
    near $+\infty$, or $H_2^+$ near $-\infty$.  
\end{proof}

We conclude this section by a characterization of the solution associated to $\HTreg$ in the
non-convex case.

\begin{proposition}\label{sub-alw-max} 
    Under the assumptions of Proposition~\ref{kc-fl}, an \usc function $u$ is an Ishii subsolution
    of \HJgen if and only if it is a \FLSub associated to the flux~limiter $\HTreg$.
\end{proposition}

\begin{proof}
    Of course, we are just interested in the inequalities on $\H\times(0,\Tf)$.

    \medskip

    \noindent\textbf{(a)} 
    We first show that an Ishii subsolution of \HJgen is necessarily a subsolution of \HJgen-\FL for
    the flux~limiter $\HTreg$.  

    Let $u$ be an Ishii subsolution of \HJgen-\FL; by Proposition~\ref{sub-up-to-b}, we already know
    that the $H_1^+$ and $H_2^-$ inequalities hold on $\H\times(0,\Tf)$ and therefore we have just
    to check the $\HTreg$-one. 

    To do so, we pick a test-function $\psi: \R^{N-1}\times (0,\Tf) \to
    \R$ and assume that $x'\mapsto u((x',0),t)-\psi(x',t)$ has a strict, local maximum point at
    $(\xb,\tb) =((\xb',0),t) \in \H\times (0,\Tf)$. Then, for $0<\e\ll 1$, we consider the function 
    $$ 
        (x,t)=((x',x_N),t)\mapsto u(x,t)-\psi(x',t) - \lambda x_N -\frac{x_N^2}{\e^2}\;,    
    $$
    where $\lambda\in[\nu_1,\nu_2]$ is fixed, $\nu_1,\nu_2$ being defined in Lemma~\ref{lem:H1m.H2p.a} at the
    point $(\xb,\tb)$ with $p'=D_{x'}\psi(\xb,\tb)$. This function has a local maximum point at a
    point $(\xe,\te)$ which converges to $(\xb,\tb)$.

    If $(\xe,\te)\in \Omega_1\times (0,\Tf)$, it follows that 
    $$ 
        \psi_t (\xe,\te) + H_1 \Big(\xe,\te, D_{x'}\psi(\xe,\te)+\lambda e_N + 
        \frac{2x_N}{\e^2}e_N \Big)\leq 0\;.
    $$
    Using that $H_1 \geq H_1^-$, the monotonicity property of $H_1^-$ (which allows to drop the
    $2x_N\e^{-2}$-term), together with the continuity of both $H_1^-$ and the derivatives of $\psi$, 
    we obtain
    $$ \psi_t (\xb,\tb) + H_1 ^-(\xb,\tb, D_{x'}\psi(\xb,\tb)+\lambda e_N )\leq o_\e(1)\;,$$
    and since $\lambda\in[\nu_1,\nu_2]$, we get
    $$ \psi_t (\xb,\tb) + \HTreg (\xb,\tb, D_{x'}\psi(\xb,\tb)+\lambda e_N )\leq o_\e(1)\;.$$
    The conclusion follows by letting $\e$ tend to $0$. The two other cases $(\xe,\te)\in
    \Omega_2\times (0,\Tf)$ and $(\xe,\te)\in \H\times (0,\Tf)$ can be treated similarly.

    \medskip

    \noindent\textbf{(b)}
    Conversely,  assuming that $u$ is a subsolution with the flux~limiter $\HTreg$, we have to show
    that it satisfies the right Ishii subsolution inequalites on $\H$. Let $\varphi$ be a smooth
    function and $(\xb,\tb) \in \H\times (0,\Tf)$ be a maximum point of $u-\varphi$, we have to show
    that
    \begin{equation}\label{subsol.ishii.ter} 
        \min\Big(a + H_1(\xb,\tb,p'+\lambda e_N), a + H_2(\xb,\tb,p'+\lambda e_N)\Big)\leq 0\;,
    \end{equation}
    where $\displaystyle a= \varphi_t (\xb,\tb),\ p'=D_{x'}\varphi(\xb,\tb),\ 
    \lambda=\frac{\partial \varphi}{\partial x_N}(\xb,\tb)$\;.
    Since the flux-limited condition on $\H\times(0,\Tf)$ reads
    $$\max\Big(a + H^+_1(\xb,\tb,p'+\lambda e_N)\;, 
     a + H^-_2(\xb,\tb,p'+\lambda e_N)\;,
     a + \HTreg (\xb,\tb,p'+\lambda e_N)\Big)\leq0\;,$$
    it is enough to prove  either $a + H^-_1(\xb,\tb,p'+\lambda e_N)\leq 0$ or  
    $a + H^+_2(\xb,\tb,p'+\lambda e_N)\leq 0$ in order to deduce \eqref{subsol.ishii.ter}.

    Now, if $\nu_1=\nu_1(\xb,\tb,p')$ and $\nu_2=\nu_2(\xb,\tb,p')$ are given by
    Lemma~\ref{lem:H1m.H2p.a},
    the result is obvious if $\nu_1\leq \lambda \leq \nu_2$. On the other hand, if $\lambda < \nu_1$, 
    $$a+H^-_1(\xb,\tb,p'+\lambda e_N)\leq a+H^-_1(\xb,\tb,p'+\nu_1 e_N)=a+\HTreg (\xb,\tb,p'+\lambda
    e_N) \leq 0\;,$$
    while if $\lambda > \nu_2$,
    $$a + H^+_2 (\xb,\tb,p'+\lambda e_N)\leq a + H^+_2 (\xb,\tb,p'+\nu_2 e_N)=
    a + \HTreg (\xb,\tb,p'+\lambda e_N) \leq 0\;.$$
    Hence in any case, \eqref{subsol.ishii.ter} holds and the proof is complete.
\end{proof}

\section{General Kirchhoff conditions and flux~limiters}

The aim of this section is to give an extension of Proposition~\ref{kc-fl} to the case of general
Kirchhoff conditions. The identification of the flux-limited condition leads to a \GFL given
by the function
\begin{equation}\label{flux.limiter.A}
    \begin{aligned}
        & A(x,t,a,p'):=\min_{s_1,s_2} \Phi(s_1,s_2) = \max_{s_1,s_2} \tilde  \Phi(s_1,s_2) \quad \text{where} \\
        & \Phi(s_1,s_2):=\max \Big( \,a+ H_1^-(x,t,p'+s_1e_N)\;,
         a+H_2^+(x,t,p'+s_2e_N)\;, G(x,t, a,p',-s_1,s_2)\,\Big)\;,\\
         & \tilde \Phi(s_1,s_2):=\min \Big( \,a+ H_1^-(x,t,p'+s_1e_N)\;,
         a+H_2^+(x,t,p'+s_2e_N)\;, G(x,t, a,p',-s_1,s_2)\,\Big)\;.
    \end{aligned}
\end{equation}
\begin{theorem}\label{thm:gkc-fl}
    Assume \GAQC and that $G$ is a \GJC of Kirchhoff type. Then $u$ is a regular \JVSub \resp \JVSup of
    \HJgen-\GJC if and only if it is a \FLSub \resp \FLSup of \HJgen-\GFL with general flux~limiter 
    $A(x,t,a,p')$ given by \eqref{flux.limiter.A}.
\end{theorem}

Here we face a general flux-limited condition, namely
\begin{equation}\label{eqn:A-FL}
    A(x,t,u_t,D_\H u)=0 \quad \hbox{on }\H\times (0,\Tf)\; .
\end{equation}
and to show that we have indeed a \GFL, we prove below that we are in the framework described in
Section~\ref{sect:types.junctions}, \ie \eqref{basic-hyp-GFL} holds.
 
\begin{proof} 
    First we leave out the proof of \FLSub \resp{\FLSup} implies \JVSub \resp{\JVSup} since, as in
    the proof of Proposition~\ref{kc-fl}, it relies on easy manipulations of the definitions.

    On the other hand, since in all the proof, the dependence in $x$, $t$, $p'$ does not play a
    role, we drop these arguments in $H_1$, $H_2$ and $G$. In other words, we essentially provide
    the proof in dimension $1$ because there is no additional difficulty in higher dimension. Notice
    however that these dependences may generate some smaller terms $o_\e(1)$ as $\e\to0$ below. 

    \medskip

    \noindent\textbf{(a)} \emph{Subsolution case ---} If $u$ is a \JVSub for the generalized
    Kirchhoff condition $G$, we have to show that it is a \FLSub with the general flux~limiter $A$, \ie if
    $\varphi=(\varphi_1,\varphi_2) \in \PC1$ and $(x,t)=((x',0),t)$ is a strict local maximum point
    of $u-\varphi$ then setting 
    $$ a=\varphi_t (x,t)\; , \; p_1=\frac{\partial
    \varphi_1}{\partial x_N}(x,t)\; ,\; p_2=\frac{\partial \varphi_2}{\partial x_N}(xs,t)\; ,$$
    we have to deduce that $ \max(a+ H_1^+(p_1), a+ H_2^-(p_2),A(a))\leq 0$ from the the \JVSub
    property, namely $$ \min(a+ H_1(p_1), a+ H_2(p_2),G(a,-p_1,p_2))\leq 0\; .$$
    The inequalities $a+ H_1^+(p_1)\leq 0, a+ H_2^-(p_2)\leq 0$ are direct consequences of
    Proposition~\ref{sub-up-to-b}, therefore we have just to show that $A(a) \leq 0$.

    Let us assume
    by contradiction that $A(a)>0$ and denoting by 
    $$ f(t)= a+H_1^-(t)\; ,\; g(s)=a+H_2^+(s)\; ,\; h(t,s) =G(a,-t,s)\;,$$
    let us use Lemma~\ref{lem:fgh}, which states that
    $$ A(a)=\min_{t,s}\left\{\max(f(t), g(s), h(t,s))\right\}=
    \max_{t,s}\left\{\min(f(t), g(s), h(t,s))\right\}\;,$$
    that both the $\min$ and $\max$ are achieved at the same point which we denote by  
    $(\bar p_1,\bar p_2)$, and finally that $A(a) = a+H_1^-(\bar p_1)=a+H_2^+(\bar p_2)=G(a,-\bar
    p_1,\bar p_2)\;.$

    We now consider the PC$^1$-function
    $$ \psi(y_N):= \begin{cases}
     \bar p_1 y_N & \mbox{ if }  y_N \geq 0\;,\\
     \bar p_2 y_N & \mbox{ if }  y_N \leq  0\;,
    \end{cases}
    $$
    and we look at maximum points of 
    $$ \chi(y,s)=u(y,s)-\varphi((y',0),s)-\psi (y_N)-\frac{y_N^2}{\eps^2}\; .$$
    Since on $\H\times (0,\Tf)$, $(x,t)$ is  a strict local maximum point of
    $u(y,t)-\varphi((y',0),s)$, there exists a sequence $(\ye,\se)$ of maximum points of $\chi$
    which converges to $((x',0),t)$. Examining the \JVSub inequality at $(\ye,\se)$, we see that, if
    $\ye \in \Omega_1$, then
    $$ a+H_1\Big(\bar p_1 + \frac{2(\ye)_N}{\eps^2}\Big)\leq 0\;.$$
    But, for $\eps$ small enough
    $$ a+H_1\Big(\bar p_1 + \frac{2(\ye)_N}{\eps^2}\Big)\geq 
    a+H_1^-\Big(\bar p_1 + \frac{2(\ye)_N}{\eps^2}\Big)\geq a+H_1^-(\bar p_1)+o_\eps(1) >0\; ,$$
    since $H_1^-$ is increasing in the normal direction $e_N$ and because 
    $a+H_1^-(\bar p_1)=A(a)>0$
    (we recall that the $o_\e(1)$-term reflects the dependence on $(x_\e,t_\e,p'_\e)$).

    Therefore  $\ye$ cannot be in $\Omega_1$, nor $\Omega_2$ by a similar argument using
    $H_2^+$. Hence $\ye = x$ but here also we get a contradiction: using as above that
    $H_1\geq H_1^-, H_2\geq H_2^+$ we obtain
    $$ \min(a+ H_1(\bar p_1), a+ H_2(\bar p_2),G(a,-\bar p_1,\bar p_2))=A(a)>0\;.$$
    This proves that $A(a)\leq 0$ and the proof is complete in the subsolution
    case.

    \medskip

    \noindent\textbf{(b)} \emph{Supersolution case  ---}
    If $v$ is a \JVSup for the generalized Kirchhoff condition $G$, we have to show that it is a
    \FLSup with the flux~limiter $A$, \ie if $\varphi=(\varphi_1,\varphi_2) \in \PC1$ and if
    $(x,t)=((x',0),t)$ is a strict local minimum point of $v-\varphi$ then, with the same notations
    as above, we have to deduce that $$\max(a+ H_1^+(p_1), a+ H_2^-(p_2),A(a))\geq 0\; ,$$ from the
    \JVSup property, namely
    $$ \max\Big(a+ H_1(p_1), a+ H_2(p_2),G(a,-p_1,p_2)\Big)\geq 0\;.$$

    We argue by contradiction assuming that $a+ H_1^+(p_1)<0$, $a+ H_2^-(p_2)<0$ and $A(a)<0$.
    Repeating exactly the arguments of the proof of Proposition~\ref{kc-fl}, we voluntarily shorten
    some passages below. Notice that a key ingredient in the proof is
    Proposition~\ref{sub-ineq-on-b} which describes the structure of sub and superdifferentials on
    $\H \times (0,\Tf)$, on both side. 
    
    Using the same notations as in Proposition~\ref{sub-ineq-on-b} and assuming also
    that $\tilde \lambda_1,\tilde \lambda_2$ are both finite, the arguments in Proposition~\ref{kc-fl} 
    first yield
    $$ a+H_1(\tilde \lambda_1)\geq 0\quad\hbox{and}\quad  a+H_2(\tilde \lambda_2)\geq 0\;,$$
    and then
    $$
    a+H_1^-(\tilde \lambda_1)\geq 0 \quad\hbox{and} \quad a+H_2^+(\tilde \lambda_2)\geq 0\;.
    $$
 
    Now, since
    $$ a+ H_1(\bar p_1)=a+ H_2(\bar p_2)=G(a,-\bar p_1,\bar p_2))=A(a)<0\;,$$
    we get $a+ H_1^-(\bar p_1), a+ H_2^+(\bar p_2)<0$ and therefore $\bar p_1< \tilde \lambda_1$,
    $\bar p_2> \tilde \lambda_2$.  Moreover, there exists $\bar p_1< s_1<\tilde \lambda_1$ and $
    \tilde \lambda_2 < s_2 <\bar p_2$ such that
    $$a+ H_1^-(s_1)= a+ H_2^+(s_2)=A(a)/2\;.$$
    The inequality $a+ H_1^-(s_1) > a+ H_1^-(\bar p_1)$ implies that $s_1$ belongs necessarily to
    the interval where $H_1=H_1^-$, and a similar argument being also true for $s_2$ we arrive at
    $$ a+ H_1^-(s_1)=a+ H_1(s_1)\quad\text{and}\quad a+ H_2^+(s_2)=a+ H_2(s_2)\;.$$

    But the fact that $s_1<\tilde \lambda_1$ and $ \tilde \lambda_2 < s_2$ means that $s_1,s_2$ are
    respectively in the subdifferential relatively to $\overline{Q_1}$ and $\overline{Q_2}$, hence
    $$ \max(a+H_1(s_1),a+ H_2(s_2),G(a,-s_1,s_2))\geq 0\;.$$
    However, each terms of the $\max$ is strictly negative: this is clear for the two first ones,
    and for the last one we use that, by the monotonicity properties of $G$, 
    $$ G(a,-s_1,s_2) \leq G(a,-\bar p_1,\bar p_2))=A(a)<0\;.$$
    So, we reach a contradiction and the proof is then complete.
\end{proof}

Now we show that the function $ A(x,t,a,p')$ given by Theorem~\ref{thm:gkc-fl} is equivalent to a
\FL condition, since it is strictly monotone in $a$.

\begin{proposition}
    Under the assumptions of Theorem~\ref{thm:gkc-fl}, there exists $\bar \gamma>0$ such that, for
    any $x\in \H, t\in [0,\Tf],p'\in \R^{N-1}$ and $a_2>a_1$ 
    $$A(x,t,a,p') - A(x,t,a,p')\geq \bar \gamma(a_2-a_1)\; .$$
    Moreover, junction condition \eqref{eqn:A-FL} is equivalent to \FL for a
    function $\G$ which satisfies  \HBAHJ.  
\end{proposition}

\begin{proof} 
    In order to prove the first part of the result, we drop the variable $x,t,p'$ which are fixed
    for the sake of simplicity of notations and therefore we assume that $H_1^-(x,t,p'+s_1e_N),
    H_2^+(x,t,p'+s_2e_N),G(x,t, a,p',-s_1,s_2)$ and $A$ are functions of $s_1, s_2$ and $a$ only.

    \medskip

    \noindent\textbf{(a)}
    By Lemma~\ref{lem:fgh}, for any $a\in \R$ there exists $s_1(a), s_2(a)$
    such that 
    $$A(a)=a+ H_1^-(s_1(a))= a+ H_2^+(s_2(a))=G(a,-s_1(a),s_2(a))\; .$$ 
    In fact, this lemma does not apply readily since $H_1^-$ is not increasing but only
    non-decreasing and $H_2^+$ is not decreasing but only non-increasing. However, this property
    remains true by easy approximations arguments, using the linear growth of $H_1^-$ at $+\infty$
    and $H_2^+$ at $-\infty$ coming from \NCe, to keep $s_1(a),s_2(a)$ bounded.

     Examining $A(a_2)-A(a_1)$ there are three cases.

    \noindent $(i)$ If $s_1(a_2)\geq s_1(a_1)$, then 
    $$ A(a_2)-A(a_1)=a_2-a_1+ H_1^-(s_1(a_2))-H_1^-(s_1(a_1))\geq a_2-a_1\; ,$$
    since $H_1^-$ is non-decreasing and the desired property is satisfied with $\bar \gamma=1$.

    \noindent$(ii)$ If $s_2(a_2)\geq s_2(a_1)$, then 
    $$ A(a_2)-A(a_1)=a_2-a_1+ H_2^+(s_2(a_2))-H_2^+(s_1(a_1))\geq a_2-a_1\; ,$$
    since $H_2^+$ is non-decreasing and the desired property is satisfied with $\bar \gamma=1$.

    \noindent$(iii)$ If $s_1(a_2)< s_1(a_1)$ and $s_2(a_2)< s_2(a_1)$, then we use the three above
    representations for $A(a_2),A(a_1)$: if $C$ is the Lipschitz constant of $H_1$, $H_2$ in $p$ and
    using the monotonicity of $G$ in $a,s_1,s_2$ 
    $$\begin{aligned}
    (2\alpha+C)(A(a_2)-A(a_1))= &\ \alpha\Big(a_2-a_1 + H_1^-(s_1(a_2))-H_1^-(s_1(a_1))\Big)\\
        & + \alpha\Big(a_2-a_1+ H_2^+(s_1(a_2))-H_2^+(s_1(a_1))\Big)\\
        & + C\Big(G(a_2,-s_1(a_2),s_2(a_2))-G(a_1,-s_1(a_1),s_2(a_1))\Big)\\
    \geq & \  \alpha\Big(a_2-a_1- C|s_1(a_2)-s_1(a_1)|\Big)\\
    & + \alpha\Big(a_2-a_1-C|s_1(a_2)-s_1(a_1)|\Big)\\
        & - \alpha C\Big((s_1(a_2)-s_1(a_1))+(s_2(a_2))-s_2(a_1))\Big)\\
         \geq & \ 2\alpha(a_2-a_1)\; .
    \end{aligned}$$
    Gathering the three cases, we see that the result holds with $\bar \gamma=2\alpha/(2\alpha+C)$.

    \medskip

    \noindent \textbf{(b)} This monotonicity property implies that that there exists $\G(x,t,p')$
    such that 
    $$ A(x,t,a,p')=0 \quad \Leftrightarrow \quad a + \G(x,t,p')=0\; .$$
    And the fact that $\G$ satisfies \HBAHJ can easily be proved by using the definition of $A$---
    which implies that $A$ satisfies \HBAHJ---and the monotonicity of $A$ in $a$.
\end{proof}

\section{Vanishing viscosity approximation (III) }

\index{Vanishing viscosity method!via flux-limited and junction viscosity solutions}

In this section, we revisit the convergence of the vanishing viscosity method in the cases of
quasi-convex Hamiltonians.  By using the connections between flux-limited and junction viscosity
solutions of problems with \FL and \KC, we are able to obtain more general results for this type of
Hamiltonians, with more complete formulations and more natural proofs. Indeed, we can combine the
advantages of these two notions of solutions, the \JVS being more flexible in terms of stability
while more general comparison results are available for \FLS (as far as quasi-convex Hamiltonians
are concerned) since they do not require the restrictive assumption \TCs.

The result is the

\begin{theorem}\label{pro:viscous}\emph{--- Vanishing viscosity limit, third version.}\smsp
    For any $\eps>0$, let $u^\eps \in C(\R^N \times [0,\Tf))$ be a viscosity solution of
    \begin{equation}\label{pb:viscous2}
       u^\eps_t -\eps \Delta u^\eps + H (x,t,u^\eps,Du^\eps) = 0\quad\text{in}\quad\R^N \times (0,\Tf)\;,
    \end{equation}
    with the initial data
    \begin{equation}\label{pb:viscousid2}
       u^\eps(x,0) = u_ 0(x) \quad\text{in}\quad\R^N \;,
    \end{equation}
    where $H(x,t,r,p) =H_1(x,t,r,p)$ if $x \in \Omega_1$ and $H(x,t,r,p) =H_2(x,t,r,p)$ if $x\in
    \Omega_2$ and $u_0$ is bounded continuous function in $\R^N$. We assume that both Hamiltonians
    $H_1,H_2$ satisfy \GAQC.

    If the $u^\eps$ are uniformly bounded in $\R^N \times (0,\Tf)$ and $C^1$ in $x_N$ in a
    neighborhood of $\H$, then, as $\eps\to0$, the sequence $(u^\eps)_\eps$ converges locally
    uniformly in $\R^N \times (0,\Tf)$ to a continuous function $u$ which is at the same time 

\noindent $(i)$ the maximal Ishii subsolution of \eqref{pb:half-space},\\
$(ii)$ the unique \JVS of the Kirchhoff problem, \\
$(iii)$ the unique \FL associated to the flux~limiter $\HTreg$.
\end{theorem}

\begin{proof} It consists in the following steps.
    \begin{enumerate}
        \item[1.] We use the stability result of Lemma~\ref{StabK}: $\ou = \limssup u^\eps$ and $\uu =
    \limiinf u^\eps$ are respectively \JVSub and \JVSup of the Kirchhoff problem. 
\item[2.] By  Proposition~\ref{kc-fl}, $\ou$ and $\uu$ are \FLSub and \FLSup with the flux~limiter
    $G=\HTreg$.  
\item[3.] By the comparison result for \FLS in the quasi-convex setting
        (Theorem~\ref{comp-IM-nc}), $\ou \leq \uu$ in $\R^N \times [0,\Tf))$
    \item[4.] By the usual argument, we deduce that $u^\eps \to u:=\ou=\uu$ in $C(\R^N \times [0,\Tf))$.
    \end{enumerate}
    We conclude the proof by remarking that Proposition~\ref{kc-fl} provides the equivalence of
    properties $(ii)$ and $(iii)$ while $(i)$ comes from the fact that an Ishii subsolution of
    \eqref{pb:half-space} is also a subsolution with \KC, hence a \FLSub with the flux~limiter
    $G=\HTreg$.  Again the comparison comes from Theorem~\ref{comp-IM-nc}.
\end{proof}

\begin{remark}
    The above proof shows how much we can take advantage of Proposition~\ref{kc-fl} and more generally of
    all the results of Chapter~\ref{sect:equiv.sols} in order to use all the different qualities of
    \FLS and \JVS.  
\end{remark}

\section{A few words about existence}
\label{sect:existence}

In general, existence of viscosity solutions is not an issue: the {\em Perron method} of
Ishii~\cite{Is-Per} (see also the User's guide \cite{Users}) provides existence of solutions in such
a general framework that addressing the question of existence has quickly become irrelevant.
On the contrary, when applying Perron method, strong comparison results are crucial in
order to obtain the existence of {\em continuous} viscosity solutions: indeed, the basic arguments
of this method consists in building an \usc subsolution $u$ such that $u_*$ is a supersolution and
then the \SCR implies the continuity of $u$ since it gives $u \leq u_*$, hence $u=u_*$ since of
course $u_*\leq u$ by definition. Therefore $u=u_*$ is both \usc and \lsc, hence continuous. Of course,
this general argument is valid for equations with discontinuous Hamiltonians (or with junctions),
which yields another reason why it is important to extend such \SCR to more and more general
contexts.

As we know, \SCR holds both for \FLS and \JVS but is it so clear that the basic arguments of the
Perron method work in these frameworks? The answer is yes but with some difficulties, which is
the reason why this section exists.

To be more precise we formulate the
\begin{proposition}\label{prop:Perron}\emph{--- Existence of solutions.}
    \begin{enumerate}
    \item[$(i)$] Under the assumptions of Theorem~\ref{comp-IM-nc}, if $u_0$ is a bounded continuous
        function, there exists a unique bounded, continuous solution of \eqref{pb:half-space} with
            the flux-limited condition given by the flux~limiter $G$.
     \item[$(ii)$] Under the assumptions of Theorem~\ref{LS-CR-GC}, if $u_0$ is continuous there
         exists a unique bounded, continuous solution of \eqref{pb:half-space} both for \GJC of
            Kirchhoff type and for \FL conditions.
    \end{enumerate}
\end{proposition}

\begin{proof} 
    Here we just sketch the proof since it readily follows the ``classical Perron method''
    approach and only focus on some specificities below. To simplify the presentation, we assume
    that $u_0$ is $C^1$ with a bounded gradient: in fact, once this particular case is treated, the
    general case follows by standard approximation arguments and stability, using in a crucial way
    a \SCR to conclude.

    We first consider the \FL case and we introduce $\uu(x,t) := u_0(x) -Ct$, $\ou(x,t) :=
    u_0(x) +Ct$. If $C>0$ is large enough, these functions are respectively \FLSub and \FLSup of
    \eqref{pb:half-space}. We then introduce the function $u_\mathrm{FL}:\R^N\times [0,\Tf]\to \R$
    defined at each point $(x,t)$ by 
    $$u_\mathrm{FL}(x,t):=\sup\big\{w(x,t): \uu\leq w\leq \ou\;,\text{ $w$ is an \FLSub}\big\}\;.$$
    Similarly, we define $u_\mathrm{GJC}$ for the \GJC case and, in the rest of the proof, $u$ denotes
    either $u_\mathrm{FL}$ or $u_\mathrm{GJC}$ since many arguments work equally for both. Notice
    that the subsolution property is checked using $u^*$ and the supersolution uses $u_*$ because
    $u$ is not continuous a priori.

    \medskip

    \noindent\textbf{(a)} \emph{The subsolution property ---} This part is easy and follows the
    standard procedure, whether in the \FL or \JVS case. It is done in three steps
    \begin{enumerate}
    \item The maximum of two subsolutions is a subsolution: a result which does not cause any
        problem in the discontinuous framework using the following property which is analogous to
            the one given in Lemma~\ref{subdiff-A}: for any \usc functions $u_1,u_2:\R^N\times
            [0,\Tf]\to \R$, for any $(x,t)\in \H \times (0,\Tf)$ such that
            $u_1(x,t)=u_2(x,t)$ and $i=1,2$ we have
        $$D_{\Omegb_i \times [0,\Tf]}^+ \max(u_1,u_2) (x,t) \subset
            D_{\Omegb_i \times [0,\Tf]}^+ u_1 (x,t) \cap D_{\Omegb_i \times
            [0,\Tf]}^+ u_2 (x,t)\; .$$
    A similar property holds if $(x,t)\in \Omega_1 \times (0,\Tf)$ or $(x,t)\in \Omega_2 \times (0,\Tf)$. 
    \item The supremum of a countable number of subsolutions is a subsolution: this is a consequence
    of Theorem~\ref{FL-stab} or Theorem~\ref{JV-stab}. Indeed, if $(u_n)_n$ is a sequence of \usc
            subsolutions\footnote{We may assume that they are \usc by replacing $u_n$ by $u_n^*$ if
            necessary.} then $v_n:=\max_{k\leq n} u_k$ is a sequence  of subsolutions by Point~1.
            Then, it is a simple exercice to show that since $(v_n)$ is non-decreasing, 
            setting $u:=\sup_{n\geq0} v_n$ yields
            $u^*=\limsup^*_{n} v_n$, were we recall that the relaxed limsup is given by
            $$\limsup\nolimits_n^* v_n=\limsup_{n\to\infty\atop (y,s)\to(x,t)} v_n(y,s)\;.$$
    
    \item The supremum of any set of subsolutions (possibly not countable) is a subsolution: indeed,
        for each $(x,t) \in \R^N\times [0,\Tf]$, there exists a sequence $(u_n)_n=(u_n^{(x,t)})_n$
            of subsolutions, whether \FLSub or \JVSub, such that $u^*(x,t)=\limsup^*_{n} u_n(x,t)$. 
            For this specific sequence $(u_n^{(x,t)})_n$, if we set 
            $$\tilde u(y,s):= \limsup\nolimits_{n}^* u_n^{(x,t)}(y,s)\;,\quad 
            (y,s)\in\R^N\times[0,\Tf]\;,$$
            the following holds: $(i)$ $\tilde u$ is a subsolution by point
            2.; $(ii)$ $\tilde u \leq u^*$ everywhere and $u^*(x,t)=\tilde u (x,t)$; 
           $(iii)$ by a similar property as the one used in point 1., 
           $$D_{\Omegb_i\times [0,\Tf]}^+ u^* (x,t)\subset 
            D_{\Omegb_i \times [0,\Tf]}^+ \tilde u (x,t)\;,\quad
            \text{for any }(x,t) \in  \Omegb_i \times [0,\Tf]\;.$$ 
            Hence the subsolution property of $\tilde u$ is
            automatically transfered to $u$.  
    \end{enumerate}
    As a by-product of the above arguments, $u^*$ is a subsolution which satisfies $\uu\leq u^* \leq \ou
    $, hence $u\geq u^*$, which means that $u=u^*$, \ie $u$ is \usc.

    \medskip

    \noindent\textbf{(b)} \emph{The \JVS case ---}  Proving that 
    the maximal subsolution $u$ is also a supersolution is done via a ``bump function'' argument.
    The reader can easily check that this argument applies without any difficulty in the case of
    $(ii)$, \ie for \JVS, when the junction condition is of Kirchhoff type. 

    The reason is the following: if $u_*$ is not a supersolution, this is
    of course because of the junction condition. Indeed, elsewhere classical Ishii's arguments apply.  
    This means that there exist $(x,t) \in \H \times (0,\Tf)$ and a test-function $\psi
    =(\psi_1,\psi_2) \in \PC1$ such that $u_*-\psi$ has a strict local minimum point at $(x,t)$ and
    $$ \max (\psi_t+H_1(x,t,u_*,D_x\psi_1), \psi_t+H_2(x,t,u_*,D_x\psi_2), G(\cdots))<0\; ,$$
    where all functions are evaluated at $(x,t)$ and $G$ at $x,t,u_*(x,t),\psi_t(x,t),D_\H
    \psi(x,t)$, $\frac{\partial \psi_1}{\partial n_1}(x,t), \frac{\partial \psi_2}{\partial
    n_2}(x,t)$. We may also assume that $u_*(x,t)=\psi(x,t)$.

     The first consequence of this property is that $u_*(x,t)< \ou(x,t)$; otherwise, $u_*$ would
     satisfy the supersolution requirement at $(x,t)$ by the same argument as Point~3. above since we
     would have $u_*\leq \ou$ and $u_*(x,t)= \ou(x,t)$, hence, for $i=1,2$
     $$D_{\Omegb_i\times [0,\Tf]}^- u_* (x,t) \subset
    D_{\Omegb_i\times [0,\Tf]}^+ \ou (x,t) \; .$$

    The second consequence is that that $\psi$ is a \JVSub in a neighborhood of $(x,t)$ since in
    particular $$\psi_t+H_1(x,t,u_*,D_x\psi_1)<0\quad\hbox{and}\quad
    \psi_t+H_2(x,t,u_*,D_x\psi_2)<0\; ,$$ the fact that $G<0$ giving the subsolution property on
    $\H\times(0,\Tf)$. Hence, using also the strict minimum point property, there exists a small
    neighborhood $\VV$ of $(x,t)$ such that, for $\e>0$ small enough, $\psi+\e$ is a \JVSub in $\VV$
    and $\psi+\e < u$ in a neighborhood of $\partial \VV$. If we set $\ue:=\max(u,\psi+\e)$ in $\VV$
    and $\ue=u$ on the complementary of $\VV$, then $\ue$ is a \JVSub  and, for $\e$
    small enough, we have $\uu \leq \ue \leq \ou$. 
    
    To get a contradiction, we have to show that there exists at least one point $(y,s)$ where $
    \ue(y,s)>u(y,s)$ since this will be a contradiction with the definition of $u$. But, by
    definition of $u_*$, there exists a sequence $(y_k,s_k)_k$ 
    converging to $(x,t)$ such that $u(y_k,s_k)\to u_*(x,t)=\psi(x,t)$. Hence
    $$ u(y_k,s_k)-(\psi+\e)(y_k,s_k)\to -\e <0\; ,$$
    and therefore $u(y_k,s_k)<(\psi+\e)(y_k,s_k)$ if $k$ is large enough. Finally
    $\ue(y_k,s_k)=(\psi+\e)(y_k,s_k)>u(y_k,s_k)$, a contradiction which implies that $u_*$ is a
    supersolution. 

    Finally, since subsolutions are regular when \GJC is of Kirchhoff type---\cf
    Proposition~\ref{prop:JVreg}---, Theorem~\ref{LS-CR-GC} shows that $u\leq u_*$ in $\R^N\times
    [0,\Tf]$, proving the continuity of $u$ and showing that $u$ is the unique solution of
    \eqref{pb:half-space} with the \GJC junction condition.

    \medskip

    \noindent\textbf{(c)} \emph{The \FL case ---} On the contrary, in case $(i)$ of
    the argument by contradiction leads to
    $$ \max (\psi_t+H_1^+(x,t,u_*,D_x\psi_1), \psi_t+H_2^-(x,t,u_*,D_x\psi_2), G(\cdots))<0\; ,$$
    which does not imply the same $H_1,H_2$ inequalities. In other words, it is not clear that $\psi$
    is a subsolution in a neighborhood of $(x,t)$ and therefore we cannot apply the ``bump
    function'' argument directly.

    To turn around this difficulty we use Proposition~\ref{prop:equivFL-LS} back and forth, being a
    little bit careful with the regularity. Since $u=u_\mathrm{FL}$ is a \FLSub, it is regular on
    $\H \times (0,\Tf)$ and therefore it is a \JVSub for the \FL condition. The ``bump function''
    argument, used exactly in the same way as above in the \JVS formulation, shows that $u_*$ is
    also a \JVSup for the \FL condition. Indeed, this argument consists in building a \JVSub which
    is strictly larger that $u$ at some point and the construction {\em preserves the regularity of
    subsolutions}. Hence this \JVSub is also a \FLSub by Proposition~\ref{prop:equivFL-LS}.

    By the same argument as above, this shows that $u$ is a continuous \JVS of \eqref{pb:half-space}
    with the \FL junction condition (by Theorem~\ref{LS-CR-GC}), hence a continuous \FLS by applying
    again Proposition~\ref{prop:equivFL-LS}.
\end{proof}

\section{Where the equivalence helps to pass to the limit}
\label{sec:using-equiv}

The aim of this section is to describe an example where using at the same time several notions of solutions
helps to pass to the limit in an asymptotic problem.

To fix ideas and to simplify matters, we consider an example which looks like the one we study in Part~\ref{part:codim1}
but with two ``close'' hyperplanes instead of one. The reader may have in mind a control problem where we only allow
regular strategies on one of the hyperplanes and all the strategies, including singular ones, on the other one. But, in the
sequel, we consider general flux~limiter on each hyperplane.

In terms of pdes, for $0<\e \ll 1$, we consider the solution $\ue \in C(\R^N \times [0,\Tf])$ of 
$$ u_t + H_2(x,t,\ue,D_x \ue)= 0 \quad \hbox{in }\{x_N<-\e\} \times (0,\Tf) \; ,$$
$$ u_t + H_0(x,t,\ue,D_x \ue)= 0 \quad \hbox{in }\{-\e<x_N<\e\}\times (0,\Tf) \; ,$$
$$ u_t + H_1(x,t,\ue,D_x \ue)= 0 \quad \hbox{in }\{x_N>\e\}\times (0,\Tf)\; ,$$
with a flux~limiter $G_2$ on the hyperplane $\{x_N=-\e\}$ and $G_1$ on the hyperplane $\{x_N=\e\}$. Taking into account
the results and methods of Chapter~\ref{chap:FLSP}, both the pde and control ones, using also the equivalence results of Chapter~\ref{sect:equiv.sols}, the associated value function is the unique \FLS or \JVS solution of the problem with the flux~limiters
$G_1$ and $G_2$. We point out that most of the arguments being local, in particular the \LCR, taking into account these two hyperplanes case is not more difficult than to consider only one hyperplane.

Our result is the following
\begin{proposition}\label{prop:pass-using-equiv} Assume that $H_0, H_1,H_2$ satisfy \HBAHJp and \NCHJ and $G_1,G_2$ satisfy
\GAGFL. Then $\ue$ converges locally uniformly to the unique solution $u$ of \HJgen-\FL with the flux~limiter
$G:=\max(G_1,G_2,(H_0)_T)$ where 
$$ (H_0)_T(x,t,r,p')= \min_{s\in\R} H_0 (x,t,r,p'+s e_N)\; .$$
\end{proposition}

\begin{proof} We first recall that, by Proposition~\ref{prop:equivFL-LS}, $\ue$ is either a \FLS or \JVS solution of the associated
flux~limiter problem and the natural idea is to use the half-relaxed limits method for the \JVS formulation which has the most general
and flexible stability result. If $\ou =\limssup \ue$ and $\uu =\limiinf \ue$, we easily obtain the $H_2$-inequality in
$\Omega_2\times (0,\Tf)$, the $H_1$-inequality in $\Omega_1\times (0,\Tf)$ and, dropping the arguments in the Hamiltonians for the sake of notational simplicity
$$ \min(\ou_t+H_0,\ou_t+H_1,\ou_t+H_2,\ou_t+G_1,\ou_t+G_2)\leq 0\; ,$$
$$ \max(\uu_t+H_0,\uu_t+H_1,\uu_t+H_2,\uu_t+G_1,\uu_t+G_2)\geq 0\; .$$
But none of these inequalities is satisfactory since they are very far from the result we wish to prove. In particular, using the normal controllability,
the first one implies
$$ \min(\ou_t+G_1,\ou_t+G_2)\leq 0\; ,$$
while we need (at least) a $\max$.

To improve these results, we first consider the case of $\ou$. We suppose that $(\xb,\tb) \in \H \times (0,\Tf)$ is a strict local maximum
point of $u-\varphi$ where $\varphi \in \PC1$. We are going to consider, for $C>0$, the following functions
$$ (x,t) \mapsto \ue (x,t)- \varphi ((x',x_N+\e),t) - C|x_N+\e|\; ,$$
$$ (x,t) \mapsto \ue (x,t)- \varphi ((x',x_N),t) - C|x_N|\; ,$$
$$ (x,t) \mapsto \ue (x,t)- \varphi ((x',x_N-\e),t) - C|x_N-\e|\; .$$
For each of these functions, there exists a subsequence $(x_{\e'},t_{\e'})$ of maximum points converging to $(\xb,\tb)$ such that $u_{\e'}
(x_{\e'},t_{\e'})\to \ou(\xb,\tb)$. Now we examine the possible viscosity inequalities at $(x_{\e'},t_{\e'})$ and to do so, we use that $u_{\e'}$ 
is a \FLS subsolution on the hyperplanes $\{x_N=-\e\}$ and $\{x_N=\e\}$ but also on the hyperplane $\{x_N=0\}$ with the flux~limiter
$(H_0)_T$ by Proposition~\ref{prop:CSasFLS}. 

By the normal controllability, if we choose $C$ large enough, it is clear that, for the first function, $x_{\e'}$ is necessarily on
$\{x_N=-\e\}$ and the $G_2$-inequality holds, while for the second one, $x_{\e'}$ is necessarily on $\{x_N=0\}$ and the $(H_0)_T$-inequality holds, and the third one leads to the 
$G_1$-inequality. Hence
$$ \max(\ou_t+G_1,\ou_t+G_2, \ou_t +(H_0)_T)\leq 0\quad  \hbox{on  } \H \times (0,\Tf).$$

The next step consists in proving that $\ou$ is regular on $\H \times (0,\Tf)$: indeed this information is crucial, on one hand, to show
that the $H_2^-$ and $H_1^+$ inequalities hold by using Proposition~\ref{sub-up-to-b} and, on the other hand, to be able to use 
Theorem~\ref{LS-CR-GC} later to get the full result.

 If this is not the case, there exists $(\xb,\tb)\in \H \times (0,\Tf)$ such that $$ \hbox{either}\quad\ou (\xb,\tb)> \limsup_{\substack{(y,s)\to (\xb,\tb)\\(y,s)\in \Omega_1 \times (0,\Tf)}} \ou(y,s)\quad \hbox{or}\quad \ou (\xb,\tb)\geq  \limsup_{\substack{(y,s)\to (\xb,\tb)\\(y,s)\in \Omega_2 \times (0,\Tf)}} \ou(y,s)\; .$$
 We assume, for example, that
$\displaystyle \ou (\xb,\tb)\geq  \limsup_{\substack{(y,s)\to (\xb,\tb)\\(y,s)\in \Omega_1 \times (0,\Tf)}} \ou(y,s)+\eta$ for some $\eta>0$, the other case being treated similarly.

For $0<\beta\ll 1$ and some large $C>0$, we introduce the function 
$$\psi_{\beta,C}(y,s) = \ou(x,t) -\frac{|x-\xb|^2}{\beta}-\frac{|t-\tb|^2}{\beta}+Cx_N\; .$$
We first consider this function in $\Omegb_2 \times (0,\Tf)$: if $\beta$ is small enough, $\psi_{\beta,C}$ achieves its maximum  at some point $(x_\beta,t_\beta)$ close to $(\xb,\tb)$ and, if $C$ is chosen large enough compared to $\beta^{-1}$, we even have $(x_\beta,t_\beta) \in \H \times (0,\Tf)$ by the normal controllability assumption because the $H_2$ inequality cannot hold. And, by subtracting a term like $|x-x_\beta|^2+|t-t_\beta|^2$, we can even assume that it is a strict 
local maximum point in $\Omegb_2 \times (0,\Tf)$.

On the other hand, if $(y,s) \in \Omega_1 \times (0,\Tf)$ is close enough to $(x_\beta,t_\beta)$, hence to $
(\xb,\tb)$, we have
\begin{align}
\psi_{\beta,C}(y,s)& =\ou(y,s) -\frac{|y-\xb|^2}{\beta}-\frac{|s-\tb|^2}{\beta}+Cy_N \\
&\leq \ou(\xb,\tb)-\frac{\eta}2 + Cy_N \\
& < \ou(\xb,\tb) = \psi_{\beta,C}(\xb,\tb)\quad \hbox{if  }Cy_N< \eta/2\; ,
\end{align}
and therefore $\psi_{\beta,C}(y,s) < \psi_{\beta,C}(x_\beta,t_\beta)$.
Hence $(x_\beta,t_\beta)$ is a strict  local maximum point in $\R^N \times (0,\Tf)$

Now, for fixed $\beta$ and $C$, we consider the functions $\psi_\e (y,s) := \ue(x,t) -\frac{|x-\xb|^2}{\beta}-\frac{|t-\tb|^2}{\beta}+Cx_N$: there exists a subsequence $(x_{\e'},t_{\e'})$ of 
maximum points of $\psi_{\e'}$ converging to $(\xb,\tb)$ such that $u_{\e'}(x_{\e'},t_{\e'})\to \ou(\xb,\tb)$. If $C$ is chosen large enough compared to $\beta^{-1}$, a case-by-case study, using the
\FLS formulation and the normal controllability, leads to a contradiction since no subsolution inequality can hold at $(x_{\e'},t_{\e'})$ if $\e'$
is small enough\footnote{in order to have $C(x_{\e'})_N< \eta/2$.}, wherever 
$x_{\e'}$ is because of the coercivity of the $H_i$'s or the fact that the $H_i^+$ are positive thanks to the $-Ce_N$-term in the derivative of 
the test-function. This shows that we cannot have $\ou (\xb,\tb)> \limsup_{\substack{(y,s)\to (\xb,\tb)\\(y,s)\in \Omega_1 \times (0,\Tf)}} 
\ou(y,s)$. The proof showing that we cannot have $\ou (\xb,\tb)> \limsup_{\substack{(y,s)\to (\xb,\tb)\\(y,s)\in \Omega_2 \times (0,\Tf)}} 
\ou(y,s)$ can be done analogously and the proof of the regularity is complete.

As we explain it above, this implies that $\ou$ is a \JVS-Sub with the flux~limiter $\max(G_1,G_2,(H_0)_T)$ and the proof for $\ou$ is
complete.

Now we turn to the supersolution properties for $\uu$. We have to prove that $\uu$ satisfies
$$ \max(\uu_t+H_1,\uu_t+H_2,\uu_t+G)\geq 0\; .$$
As above, we suppose that $(\xb,\tb) \in \H \times (0,\Tf)$ is a strict local minimum point
of $u-\varphi$ where $\varphi =(\varphi_1,\varphi_2) \in \PC1$. We argue by contradiction assuming that
$$ \max(\varphi_t+H_1,\varphi_t+H_2,\varphi_t+G)=-\eta< 0\; .$$
We consider the function
$$(x,t) \mapsto \ue(x,t)-\varphi(x,t)-\e \chi\left(\frac{x_N}\e\right)\; ,$$
where $\chi:\R\to \R$ is defined in the following way
$$ \chi(y)=\begin{cases}
-\delta_2 & \hbox{if }y\leq -1\; , \\
\delta_2 y & \hbox{if }-1 \leq y \leq 0\; , \\
\delta_1 y & \hbox{if }0 \leq y \leq 1\; ,\\
\delta_1 & \hbox{if }y \geq 1\; ,
\end{cases}
$$
where $\delta_1, \delta_2$ will be chosen later on.

As above, there exists a subsequence $(x_{\e'},t_{\e'})$ of minimum points of this function converging to $(\xb,\tb)$ such that $u_{\e'}(x_{\e'},t_{\e'})\to \uu(\xb,\tb)$. In order to examine the possible viscosity inequalities at $(x_{\e'},t_{\e'})$, we set for $F=H_0,H_1,H_2,G_1,G_2$
$$ \tilde F(\tau):=\varphi_t(\xb,\tb)+ F(\xb,\tb,\uu(\xb,\tb), D_{x'}\varphi(\xb,\tb) +\tau e_N)\; .$$

By assumption, we have
$$\tilde H_1(\dfrac{\partial \varphi_1}{\partial x_N})\leq -\eta <0\quad,\quad  \tilde H_2(\dfrac{\partial \varphi_2}{\partial x_N})\leq -\eta <0\; ,$$
and the constants (in $\tau$) $\tilde G_1, \tilde G_2,  (\tilde H_0)_T$ are also less than $-\eta<0$ .

Now we examine the different possibilities
\begin{enumerate}
\item[(a)] $(x_{\e'})_N <-\e'$: then, by the continuity of $H_2$ and the fact that $\varphi_2$ is $C^1$, we should have $\tilde H_2(\dfrac{\partial \varphi_2}{\partial x_N}) \geq o(1)$ but clearly this inequality cannot hold for $\e'$ small enough.
\item [(b)] $(x_{\e'})_N =-\e'$: using again the continuity of the Hamiltonians and of $\varphi_2$, the \FLSup inequality should read
$$ \max ((\tilde H_2)^- (\dfrac{\partial \varphi_2}{\partial x_N}), (\tilde H_0)^+ (\dfrac{\partial \varphi_2}{\partial x_N}+\delta_2), \tilde G_2)\geq o(1)\; .$$
Here $\tilde G_2<0$, $(\tilde H_2)^- \leq \tilde H_2<0$ and we choose $\delta_2$ in order that $\dfrac{\partial \varphi_2}{\partial x_N}+\delta_2$
is a minimum point of $\tilde H_0$, hence
$$ (\tilde H_0)^+ (\dfrac{\partial \varphi_2}{\partial x_N}+\delta_2)=(\tilde H_0)_T<0\; .$$
With this choice of $\delta_2$, this second case turns out to be impossible.
\item [(c)] $-\e'<(x_{\e'})_N <0$: with our choice of $\delta_2$, the $H_0$-inequality cannot hold for $\e'$ small enough and this
case cannot happen neither.
\item [(d)] $(x_{\e'})_N =0$: we choose $\delta_1$ such that $\dfrac{\partial \varphi_1}{\partial x_N}+\delta_1$
is a minimum point of $\tilde H_0$, hence
$$ (\tilde H_0)^+ (\dfrac{\partial \varphi_1}{\partial x_N}+\delta_1)=(\tilde H_0)_T<0\; .$$
With this choice, the \FLSup inequality which reads
$$ \max((\tilde H_0)^- (\dfrac{\partial \varphi_2}{\partial x_N}+\delta_2), (\tilde H_0)^+ (\dfrac{\partial \varphi_1}{\partial x_N}+\delta_1), (\tilde H_0)_T)\geq 0\; ,$$
cannot hold for $\e'$ small enough.
\item [(e)] $0<(x_{\e'})_N <\e'$: this case is the exact symmetric of (c),
\item [(f)] $(x_{\e'})_N =\e'$: this case is the exact symmetric of (b),
\item [(g)] $(x_{\e'})_N >\e'$: this case is the exact symmetric of (a),
\end{enumerate}  
and in the three cases (e), (f), (g), we also conclude that the \FLSup inequality cannot hold for $\e'$ small enough. Hence, wherever $x_{\e'}$ is, the \FLSup inequality cannot hold. This gives a contradiction and prove that
$$ \max(\varphi_t+H_1,\varphi_t+H_2,\varphi_t+G)\geq 0\; .$$
Hence $\uu$ is a \JVSup with the flux~limiter $G$.

The classical arguments of the half-relaxed limits method to gether with the comparison result for \JVS solutions (Theorem~\ref{LS-CR-GC}), taking into account that $\ou$ is a regular subsolution, implies $\ou \leq \uu$ in $\R^N\times [0,\Tf]$. Hence $u=\ou=\uu$ is continuous and the unique \JVS with the flux~limiter $G$. And the local uniform convergence of $\ue$ to $u$ follows by classical arguments.
\end{proof}

\chapter{Applications and Emblematic Examples}
\label{Embl-Exemple}
\abstract{This chapter gives an overview of the results of Part~\ref{part:codim1} in the context of
Hamilton-Jacobi equations corresponding to 1D scalar conservations laws with a discontinuous flux.
It is especially intended for the (partial) reader who wishes to get an idea of what can be done in
this context ... without reading the totality of this book!} 

According to the aim of this chapter given in the above abstract, the reader will find here some
redundancy concerning definitions, results, ideas... with respect to the previous sections. This is,
of course, unavoidable taking into account the objective of this chapter. On the other hand, we try
to keep it as simple as possible and refer to those previous sections for more precise results and
proofs.

\section{HJ analogue of a discontinuous 1D-scalar conservation law}
\index{Scalar conservation laws}

The starting point here is the problem
\begin{equation}\label{HJ-d1}
u_t + H(x,u_x)=0\quad\hbox{in }\R \times (0,\Tf)\; ,
\end{equation}
where Hamiltonian $H$ is given by
$$
H(x,p)=\begin{cases}
H_1(p) & \hbox{if $x>0$\;,}\\
H_2(p) & \hbox{if $x<0$\;.}
\end{cases}
$$
Equation~\eqref{HJ-d1} has to be complemented by an initial datum
\begin{equation}\label{HJ-d1-id}
u(x,0)=\u0 (x) \quad\hbox{in }\R \; ,
\end{equation}
where $\u0$ is assumed to be bounded and continuous in $\R$.

In this definition of $H$, $H_1,H_2$ are continuous functions which are coercive, \ie
$$ H_1(p), H_2(p) \to +\infty\quad \hbox{as  }|p| \to +\infty\; ,$$
and we consider two main cases: the ``Lipschitz case'' where both Hamiltonians are supposed to be
Lipschitz continuous in $\R$ and the ``convex case'' where they are supposed to be convex, but not
necessarily Lipschitz continuous, even if this case is not completely covered by the results of
Part~\ref{part:codim1} \footnote{but we trust the reader to be able to fill up the gaps!}.

In the ``Lipschitz case'', a natural sub-case is the one when the $H_i$ $(i=1,2$) are {\em
quasi-convex} \footnote{We refer the reader to Section~\ref{sect:quasi.convexity} for a short
presentation of the notion of quasi-convexity and for the related properties we use throughout this
book.}, \ie built as the maximum of an increasing and a decreasing function.  For this reason, we
write 
$$ H_1=\max(H_1^+,H_1^-)\quad\hbox{and}\quad H_2=\max(H_2^+,H_2^-)\;,$$
where $H_1^+, H_2^+$ are the decreasing parts of $H_1,H_2$ respectively and $H_1^-, H_2^-$ their
increasing parts. Using these notations for the monotone Hamiltonians may seem strange but the
reader has to keep in mind that {\em characteristics}---or dynamics in terms of control problems---play 
a key role in these problems. A way to better understand this remark is to consider the convex
control case where the dynamic is given by $b_1$ and where $H_1(p)=\sup_{b_1\in B}\{-b_1\cdot
p-l_1\}$; in this case
$$H_1^+(p)=\sup_{b_1\geq 0}\{-b_1\cdot p-l_1\}\; ,$$
which means that we keep in $H_1^+$ only the dynamics pointing toward the positive direction,
explaining the ``$+$''.

Of course, the case of quasi-concave (or concave) Hamiltonians can be treated in the same way since,
by changing $u$ in $-u$, we change $H_1(p),H_2(p)$ in $-H_1(-p),-H_2(-p)$, the latter being
quasi-convex (or convex) if the former are quasi-concave (or concave).

\subsection{On the condition at the interface}

Of course, the first key question is: what kind of condition has to be imposed at $x=0$ where the
Hamiltonian $H$ is discontinuous? 

Viscosity solutions theory provides a default answer which is the notion of
{\em Classical Viscosity Solutions} (\CVS in short) introduced by by Ishii \cite{I1}. These conditions are
$$\min(u_t + H_1(u_x), u_t + H_2(u_x))\leq 0\quad\text{on}\quad\{0\} \times (0,\Tf)\; ,$$
$$\max(u_t + H_1(u_x), u_t + H_2(u_x))\geq 0\quad\text{on}\quad\{0\} \times (0,\Tf)\; .$$
These sub and supersolutions properties have to be tested with test-functions which
are $C^1$ in $\R \times (0,\Tf)$. We do not detail them here, referring the
reader to Section~\ref{sect:stab} for more informations.

Unfortunately (or fortunately?), this classical notion of solution has two main defects: on one
hand, \CVS are not unique in general and, on the other hand, in concrete applications,
the modelling may lead to other ``transfer conditions'' at $x=0$. To be
convinced by this claim, it suffices to look at the well-known Kirchhoff condition
\begin{equation}\label{kir-oned}
       - u_x (0^+,t) + u_x (0^-,t) = 0\quad\text{on}\quad\{0\} \times (0,\Tf)\;,
\end{equation}
for which testing with $C^1(\R\times (0,\Tf))$-test-functions is of course meaningless, this
condition being automatically satisfied for smooth test-functions. Clearly we need a larger set of
testing possibilities in order to take into account in a right way such conditions and to have a
hope for a well-posed problem (in particular, comparison and uniqueness results).

\subsection{Network viscosity solutions}

For the Kirchhoff condition but also for more general conditions like
\begin{equation}\label{JC-oned}
    G(u_t,- u_x(0^+,t) , u_x(0^-,t)) = 0\quad\text{on}\quad\{0\} \times (0,\Tf)\;,
\end{equation}
where $G(a,b,c)$ is a continuous function which is increasing in $a,b$ and $c$\footnote{Precise
assumptions will be given later on.}, one has to use a notion of {\em ``Network viscosity
solution''} based on testing the viscosity properties with continuous, ``piecewise
$C^1$''-test-functions, denoted by $\mathrm{PC}^1$. 
More precisely $\phi \in C(\R \times (0,\Tf))$ is a suitable test-function
if there exists two functions $\phi_1, \phi_2$ which are $C^1$ in $\R \times (0,\Tf)$ such that
$$
    \phi(x,t)=\begin{cases}
    \phi_1(x,t) & \hbox{if $x>0$,}\\
    \phi_2(x,t) & \hbox{if $x<0$,}
    \end{cases}
$$
with $\phi_1(0,t)=\phi_2(0,t)$ for any $t \in (0,\Tf)$. In order to define ``Network viscosity
solutions'' in the viscosity properties at a point $(0,t)$ we use the derivatives of $\phi_1$ for
the $H_1$ and $u_x(0^+,t)$-term, and the derivatives of $\phi_2$ for the $H_2$ and
$u_x(0^-,t)$-term. Notice that both time-derivatives $(\phi_1)_t$ and $(\phi_2)_t$ coincide on $x=0$.

However, the notion of ``Network viscosity solution'' with condition at $x=0$ can be used in at
least two slightly different ways. 

\medskip

\noindent\textbf{(a)} \emph{The ``flux-limited''} notion of solutions of Imbert-Monneau---\FL in
short---which is valid in the quasi-convex case, \ie in a more general framework than the ``convex
case''.  A general flux-limited condition at $x=0$ takes the form
\begin{equation}\label{FL1-oned}
u_t + A = 0\quad\text{on}\quad\{0\} \times (0,\Tf)\;,
\end{equation}
where $A$ is a real constant called the \emph{flux~limiter}. 
In terms of viscosity inequalities at $x=0$, the condition reads 
\footnote{with the above mentioned conventions}
$$\begin{cases} 
    \max\big(u_t + H_1^+ (u_x), u_t + H_2^-(u_x), u_t +A\big) \leq 0 \quad\text{on}\quad\{0\} \times (0,\Tf)\;,\\
    \max\big(u_t + H_1^+ (u_x), u_t + H_2^-(u_x), u_t +A\big) \geq 0 \quad\text{on}\quad\{0\} \times (0,\Tf)\;.
\end{cases}$$
Why using only $H_1^+$ and $H_2^-$? As we already explain it above, the most (vague but) convincing
answer is probably through the {\em characteristics}, or dynamics in the control viewpoint: we use
inequalities which test characteristics entering each domain, \ie $[0, +\infty)$ for $H_1$ and
$(-\infty,0]$ for $H_2$. We respectively call these conditions the sub and supersolution \FL
conditions. In the definition above, we can replace the $u_t +A$-term by a more general
$\chi(u_t)$-term where the function $\tau \mapsto \chi(\tau)$ is strictly increasing. 

\medskip

\noindent\textbf{(b)} \emph{The notion of ``junction viscosity solutions'' \JVS} which is closer to
the Ishii formulation since the inequalities for $x=0$ read 
$$\begin{cases}
    \min\,\big(u_t + H_1(u_x), u_t + H_2(u_x),G(u_t,- u_x(0^+,t) , u_x(0^-,t))\big)\leq 0\;,\\
    \max\big(u_t + H_1(u_x), u_t + H_2(u_x),G(u_t,- u_x(0^+,t) , u_x(0^-,t))\big)\geq 0\;.
\end{cases}$$

We refer to Section~\ref{sec:def-FLS} and Section~\ref{sec:def-JVS} for a more precise definition of
\FLS and \JVS. Let us point out three key differences between these notions of solutions:
\begin{enumerate}
    \item[$(i)$] the notion of \JVS can take into account both general Kirchhoff type conditions
        like Equation~\ref{JC-oned} but also flux-limited conditions by assuming that $G(a,b,c)=a+A$
        in Equation~\ref{JC-oned}. On the contrary, the notion of \FLS is restricted to flux-limited
        conditions;
    \item[$(ii)$] while the notion of \FL solutions requires the Hamiltonians to be quasi-convex,
        the \JVS notion is valid for any continuous Hamiltonians;
    \item[$(iii)$] while the \FLS one uses a pair of ``max-max'' inequalities, the \JVS one uses a
        classical ``min-max'' ones.
\end{enumerate}

These three differences seem to indicate that the notion of \JVS is more general and more adapted
than the \FLS one but the notion of \FLS is more natural to address control problems (see
Section~\ref{sect:control.NA}) and therefore is useful in order to obtain explicit formulas \`a la
Oleinik-Lax.

\subsection{Main results}

We now expose briefly the main results and connections between the different notions of solutions,
\CVS, \FLS, \JVS.

The convergence of the vanishing viscosity method is a natural entrance door since, in the classical
framework, it selects the ``right solution''.

\medskip

\noindent\textbf{(a)} \emph{On the vanishing viscosity method and the Kirchhoff junction solution ---} 
In the absence of discontinuities, passing to the limit in this method simply relies on the
stability properties of classical viscosity solutions. However here, in presence of discontinuous
Hamiltonians, we need to identify the right condition on the interface. The result is the
\begin{theorem}
    \label{thm:viscous2-oned}
    \emph{--- Convergence of the vanishing viscosity method.}\smsp
    For each $\eps>0$, let $u^\eps$ be a continuous viscosity solution of
    \begin{equation}\label{pb:viscous-oned}
       u^\eps_t -\eps u_{xx}^\eps + H (x,u_x^\eps) = 0\quad\text{in}\quad\R \times (0,\Tf)\;,
    \end{equation}
    associated with the initial data
    \begin{equation}\label{id:pb:viscousid-oned}
       u^\eps(x,0) = u_ 0(x) \quad\text{in}\quad\R \;.
    \end{equation}
    If the $u^\eps$ are uniformly bounded in $\R \times [0,\Tf)$ and $C^1$ in $x$ in a neighborhood
    of $x=0$ for $t>0$, then, as $\eps\to0$, the sequence $(u^\eps)_\eps$ converges locally
    uniformly to the unique \JVS solution of the Kirchhoff problem
    \eqref{HJ-d1}-\eqref{HJ-d1-id}-\eqref{kir-oned}.
\end{theorem}

\index{Kirchhoff condition}
We first point out that Theorem~\ref{thm:viscous2-oned} is valid for any continuous Hamiltonians
$H_1,H_2$ without any type of convexity (or concavity) assumption.

The formal idea to prove this result is straightforward: $u^\eps$ being $C^1$ in $x$ in
a neighborhood of $x=0$ for $t>0$, it satisfies the Kirchhoff condition
$$
- u_x^\eps(0^+,t) + u_x^\eps(0^-,t) = 0\quad\text{on}\quad\{0\} \times (0,\Tf)\;,
$$
and it suffices to pass to the limit using the good stability properties of viscosity solutions, but
written with piecewise $C^1$ test-functions and then to use an adapted comparison result. 

This formal proof can be justified using the notion of \JVS solutions via Lions-Souganidis arguments
for the comparison result. Indeed, on one hand, this notion of solutions allows to extend the
classical stability argument for viscosity solutions: the half-relaxed limits of $u^\eps$ are
``junction sub and supersolution'' of the Kirchhoff problem, \ie
$$\begin{cases}
    \min(u_t + H_1(u_x), u_t + H_2(u_x),- u_x (0^+,t) + u_x(0^-,t) )\leq 0\; ,\\
    \max(u_t + H_1(u_x), u_t + H_2(u_x),- u_x (0^+,t) + u_x(0^-,t))\geq 0\; .
\end{cases}$$
Hence we have a stability result which is as similar as it could be to the classical one, despite of
the different spaces of test-functions. It is worth pointing out that the notion of \JVS  is not
only necessary to define properly the Kirchhoff condition but it also plays a key role here via this
stability result. And then the Lions-Souganidis arguments provide the ``Strong Comparison Result''
which is needed to conclude.

We actually provide three different proofs of the convergence of the vanishing viscosity method in
Part~\ref{part:codim1}: a first one via \FLS solutions, the above one via \JVS solutions and a last
one which combines both notions in order to identify the limit, in particular by giving an explicit
formula in the control case.

\medskip

\noindent\textbf{(b)} \emph{On Classical Viscosity Solutions and the Kirchhoff condition ---} 
    The result of Theorem~\ref{thm:viscous2-oned} suggests two natural questions
\begin{enumerate}
    \item[1.] Is it possible to characterize the unique \JVS of
        \eqref{HJ-d1}-\eqref{HJ-d1-id}-\eqref{kir-oned}, \ie the Kirchhoff solution, in terms of
        classical viscosity solutions \CVS? 
    \item[2.] In the ``convex case'', is it possible to write down an explicit formula for solutions
        of the Kirchhoff problem? (\`a la Oleinik-Lax). In other words, is there an underlying control
        problem which gives a control formula for this solution?
\end{enumerate}
Our second result answers these questions, of course in the ``convex case'' since we are looking for
explicit formulas. We point out that the results are unavoidably a little bit vague to avoid long
statements but precise results can be found in Chapter~\ref{chap:Ishii}.  
\begin{theorem}\label{cont:viscous2-oned}
    \emph{--- Classical Viscosity Solutions.}\smsp
    In the ``quasi-convex or convex case'',
    \begin{enumerate}
    \item[$(i)$] Classical Viscosity Solutions of \eqref{HJ-d1}-\eqref{HJ-d1-id} with the natural
        Ishii conditions at $x=0$ are not unique in general. There is a minimal \CVS denoted by
        $\Um$ and a maximal \CVS denoted by $\Up$. In the convex case, they are both given explicitly
        as value functions of suitable control problems.
    \item[$(ii)$] If $m_1$ is the largest minimum point of $H_{1}$ and $m_2$ the least minimum of
        $H_{2}$, a sufficient condition in order to get $\Um=\Up$ is $m_2\geq m_1$.  
    \item[$(iii)$] The solution of the Kirchhoff problem is $\Up$. Hence the vanishing
        viscosity method converges to the maximal \CVS.
    \end{enumerate}
\end{theorem}

This result shows the weakness of \CVS for equations with discontinuities: although they are very
stable because of the half-relaxed limits method, they are not unique in this framework and this
is, of course, more than a problem. Result~$(ii)$ is a last desperate attempt to maintain
uniqueness in a rather general case but it seems to be a little bit anecdotic.

Result~$(iii)$ is a first bridge between the notions of \CVS and ``junction solutions'' and it is
proved using in a key way the notion of ``flux-limited solutions''.

We refer to Chapter~\ref{sect:equiv.sols} for various results comparing the \CVS, \FLS and \JVS
notions of solutions.

\medskip

\noindent\textbf{(c)} \emph{On the characterization via flux-limited solutions ---}
The previous results open the way to the next questions which can be formulated in several different
ways, but which all concern the relations between different notions of solutions.
\begin{enumerate}
    \item[1.] In the case of control problems, two particular value functions appear in
        Theorem~\ref{cont:viscous2-oned}, $\Um$ and $\Up$. Both may be interesting for some
        particular application but clearly, the characterization as \CVS is not appropriate. Is
        there any other way to identify them uniquely?
    \item[2.] From the pde point of view, Result~$(iii)$ gives a connection between the ``junction
        solution'' for the Kirchhoff condition and a value function of some control problem. But is it
        possible to prove some similar connexion for more general conditions \eqref{JC-oned} with a
        rather explicit way?
\end{enumerate}
The answer is provided in the following result which relies on the notion of  \FLS.
\begin{theorem}\label{FL-oned}
    \emph{--- Characterizations with flux~limiters.} \smsp 
    \noindent{\bf A. In the quasi-convex case}
    \begin{enumerate}[topsep=3pt]
        \item[$(i)$] For any $A$, there exists a unique flux-limited solution of
            \eqref{HJ-d1}-\eqref{HJ-d1-id}-\eqref{FL1-oned}. Moreover a comparison principle holds 
            result for this flux-limited problem.
        \item[$(ii)$] If $G$ satisfies: there exists $\alpha, \beta\geq 0$ with $\beta >0$
            such that for any $a_1\geq a_2$, $b_1\geq b_2$, $c_1\geq c_2$
            $$ G(a_1,b_1,c_1)-G(a_2,b_2,c_2)\geq \alpha(a_1-a_2) + \beta(b_1-b_2)+ \beta(c_1-c_2)\; ,$$
            then any JVS subsolution \resp{supersolution} of \eqref{HJ-d1}-\eqref{HJ-d1-id}-\eqref{JC-oned}
            is a \FLS subsolution \resp{supersolution} with flux~limiter
            $$ \chi(a)= \max_{p_1,p_2}\Big( \min\left(a+ H_1^-(p_1), 
            a+ H_2^+(p_2),G(a,-p_1,p_2)\right) \Big)\; .$$
    \end{enumerate}

    \noindent{\bf B. In the convex case}
     \begin{enumerate}[topsep=3pt]
    \item[$(i)$] The value function $\Um$ is associated to the flux~limiter
        $$ A^-=\min_s\big(\max(H_1(s),H_2(s))\big)\; .$$
    \item[$(ii)$] The value function $\Up$ is associated to the flux~limiter
        $$A^+=\min_s\big(\max(H_1^- (s),H^+ _2(s))\big)\; .$$
     \end{enumerate}
\end{theorem}

The second part of this result shows that value functions of control problems can be characterized as
a ``flux-limited solution'' of \eqref{HJ-d1}-\eqref{HJ-d1-id} with the right flux~limiter
 at $x=0$. Contrarily to \CVS, a uniqueness result holds but, as the vanishing viscosity
method shows, stability becomes a problem since one has to identify the right flux~limiter for
the limiting problem.

\begin{remark}
    The case of more general junction conditions like 
    \eqref{HJ-d1}-\eqref{HJ-d1-id}-\eqref{JC-oned} can be treated by the Lions-Souganidis approach:
    in particular, we have a comparison result for \eqref{HJ-d1}-\eqref{HJ-d1-id}-\eqref{JC-oned} in
    the case of general Hamiltonians $H_1,H_2$ without assuming them to be quasi-convex. Of course,
    the monotonicity properties of $G$ are necessary not only for having such a comparison result
    but even for the notion of ``junction solution'' to make sense.
\end{remark}

We conclude this section by the extension of the Oleinik-Lax formula to our discontinuous framework.
To do so, we set $\Omega_1=\{x>0\}$ and $\Omega_2=\{x<0\}$ and we denote by $H_i^*$ the Fenchel
conjugate of $H_i$ for $i=1,2$.  
\begin{proposition}\label{prop:OL}\emph{--- Oleinik-Lax Formula.}\smsp
    Under the assumptions of Theorem~\ref{FL-oned}, we assume moreover that $H_1,H_2$ are convex
    coervive, continuous and set, for $x\in \Omegb_i$ and $t>0$, 
    $$\mathbf{U}^i(x,t):=\inf_{z\in\Omegb_i} \left(u_0(z) +
    tH_i^*(\dfrac{x-z}{t})\right)\;.$$
    Then the following formulas hold: $\Up (x,t)$ is given by
    $$
    \min
    \left(\mathbf{U}^i(x,t) , \inf_{\substack{j=1,2,\ z\in \Omega_j \\ 0\leq t_1 \leq t_2\leq t}}\,
    \left\{ u_0(z) + t_1H_i^*(\frac{x}{t_1})-A^+(t_2-t_1)+ (t-t_2)H_j^*(\frac{-z}{t-t_2})  \right\} \right)\; ,
    $$
    while $\Um (x,t)$ is given by
    $$
     \min\left( \mathbf{U}^i(x,t), \inf_{\substack{j=1,2,\ z\in \Omega_j \\ 0\leq t_1 \leq
    t_2\leq t}}\,\left\{ u_0(z)+ t_1 H_i^*(\frac{x}{t_1})-A^-(t_2-t_1)+
    (t-t_2)H_j^*(\frac{-z}{t-t_2})  \right\} \right)\; ,
    $$
    with the convention that $(t-t_2)H_j^*((-z)/(t-t_2))=0$ if $z=0$ and $t-t_2=0$.
\end{proposition}

\medskip

In order to apply these Oleinik-Lax formulas, we come back on the examples of
Section~\ref{sec:more-examples}. In the first one, $$ H_1(p)=|p+1|\quad,\quad H_2(p)=|p-1|\; ,$$
therefore $A^+=A^-=1$ and it follows that uniqueness holds in the Ishii class of solutions, $\Up\equiv\Um$.

Now, $H^*_1(p)=-p$ if $|p|\leq 1$ and $+\infty$ otherwise, while $H^*_2(p)=p$ if
$|p|\leq 1$ and $+\infty$ otherwise. Hence, since $u_0(x)=|x|$ in $\R$, we see that
$$\begin{cases}
    \mathbf{U}^1(x,t):=\inf\limits_{z\geq 0,\ |z-x|\leq t} \left(z -(x-z)\right)=2(|x|-t)_+ -|x| \;,\\
    \mathbf{U}^2(x,t):=\inf\limits_{z\leq 0,\ |z-x|\leq t} \left(-z +(x-z)\right)=2(|x|-t)_+-|x|\;.
\end{cases}$$
Then, in order to compute $\Up=\Um$, we face several cases
\begin{enumerate}
    \item[$(i)$] If $|x|>t$, then $H_i^*(\frac{x}{t_1})=+\infty$ so 
        $\Up(x,t)=\mathbf{U}^i(x,t)=2(|x|-t)_+-|x|$.
    \item[$(ii)$] If $|x|\leq t$, for the second part of the $\min$, if $x,z\in \Omegb_1$
        we get
        $$\inf_{\substack{z\leq t-t_2, x\leq t_1 \\ 0\leq t_1 \leq t_2\leq t}}\,
        \left\{ z-x - (t_2-t_1)+z  \right\} =-t\; ,$$
        since $z=0$, $t_1=x$ and $t_2=t$ is clearly optimal.
        If $x,z\in \Omegb_2$, an analogous result holds.
    \item[$(iii)$] If $|x|\leq t$, for the second part of the $\min$, if $x\in \Omegb_1$,
        $z\in \Omegb_2$ we get 
        $$\inf_{\substack{z\leq t-t_2, x\leq t_1 \\ 0\leq t_1 \leq t_2\leq t}}\,
        \left\{ -z-x - (t_2-t_1)-z  \right\} =-t\; ,$$
        since $z=0$, $t_1=x$ and $t_2=t$ is clearly optimal.
        The case $x\in \Omegb_2$, $z\in \Omegb_1$ gives an analogous result.
\end{enumerate}
Finally, since for $|x|\leq t$, we have $2(|x|-t)_+-|x|\geq - t$, we conclude that
$$ \Up(x,t)=\Um(x,t)=\begin{cases}
2(|x|-t)_+-|x| & \hbox{if $|x|\geq t$}\; ,\\
-t & \hbox{otherwise.}
\end{cases}$$

In the second example
$$ H_1(p)=|p-1|\quad,\quad H_2(p)=|p+1|\; ,$$
and therefore $A^+=0$, $A^-=1$. This time, the solution $\Up$ and $\Um$ are different. We leave the checking of their formulas
to the reader (they are given in Section~\ref{sec:more-examples}).

\section{Traffic flow models with a fixed or moving flow constraint}
\label{Traffic} 

\index{Traffic!LWR model}

Traffic flows can be studied at the micro or macroscopic level, leading to different, yet
complementary models. Here we focus only on the macroscopic scale, looking at the density
$\rho(x,t)$ of vehicles at each point $x$ of a one-dimensional infinite highway modeled by $\R$ and
any time $t$. It is often more convenient to use the ``renormalized density", \ie the ratio between
the actual density and a maximal density, therefore assuming that $0\leq \rho(x,t)\leq 1$ for any
$x$ and $t$.

\subsection{The LWR model}
\index{Traffic!LWR model}

In the context of simple traffic flow without constraints, one of the most famous macroscopic
models is the LWR model, originated in the works of Lighthill and Whithan \cite{LiWh} and Richards
\cite{Ri}. It consists in describing the evolution of $\rho$ through a scalar conservation law in
$\R\times(0,\Tf)$, namely
\begin{equation}\label{scl-wc}
    \rho_t +\partial_x(f(\rho))=0 \quad \hbox{in  }\R \times (0,\Tf)\; ,
\end{equation}
where the flux $f:\R \to \R$ is given in the simplest case by $f(\rho)=\rho(1-\rho)$.

In one space dimension, a rather easy way to tackle Equation~\eqref{scl-wc} is to use the
connections with the Hamilton-Jacobi formulation, \cf Corrias, Falcone and Natalini~\cite{CFN} and
Aaibid and Sayah \cite {MA-AS}: if $u:\R \times (0,\Tf)\to \R$ is the unique Lipschitz continuous
viscosity solution of 
\begin{equation}\label{hj-wc}
    u_t +f(u_x)=0 \quad \hbox{in  }\R \times (0,\Tf)\; .
\end{equation}
with a Lipschitz continuous initial data $u_0$ such that $0\leq u_0'(x)\leq 1$ in $\R$, then
$\rho=u_x$ is the unique entropy solution of ~\eqref{scl-wc}.

At this point, it is worth pointing out that here $f$ is concave, not convex. However, using that
$-u$ is a solution of \eqref{hj-wc} with $f$ being replaced by $h(p)=-f(-p)$ which is convex, one
may use the Oleinik-Lax formula to obtain the explicit form of the solution:
$$ 
    u(x,t):=\sup_{y \in \R}\left\{u_0(y) -\frac t 4 
    \left\vert \frac{x-y}{t}-1\right\vert ^2 \right\}.
$$
In particular, if $u_0(y)=\rho_0y$ in $\R$ for some $0<\rho_0<1$, corresponding to a constant
density $\rho_0$ at time $t=0$, then 
\begin{equation}\label{traffic.formula}
    u(x,t)=\rho_0x-\rho_0(1-\rho_0)t \quad \hbox{for any  }(x,t)\in \R \times (0,\Tf)\; .
\end{equation}

\subsection{Constraints on the flux}
\index{Traffic!constraints on the flux}

A more interesting question in the context of this book is to investigate the question of traffic
reduction, due to a car crash or traffic lights located at one point, $x=0$. Such problems were
first studied by Colombo and Rosini \cite{CoRo} and Colombo and Goatin \cite{CoGo} and they lead to
a constraint of the type
$$f(\rho)\,\Big|_{x=0}\leq \delta \;,$$
where we we choose here to consider a {\em constant} flux~limiter $\delta>0$.

To the best of our knowledge, there is no rigorous result connecting such problems with
Hamilton-Jacobi ones in this framework. We can just guess that the corresponding constraint in the
Hamilton-Jacobi case takes the form
$$ u_t \geq -\delta \quad\hbox{at  }x=0$$
since at least formally, $f(\rho)=f(u_x)=-u_t$. This formulation corresponds to a flux~limiter
$G=+\delta$ at $x=0$ coupled with Hamiltonians $H_1=H_2=f$ but we recall that $f$ being concave we
have either to change $u$ in $-u$ or to adapt the results of the previous sections. In any case, we
point out that all these results apply and the theory can be applies without any difficulty.

Now, let us go back to the same initial data as above: $u_0(y)=\rho_0y$ in $\R$ for some $0<\rho_0<1$,
associated to a flux $\delta<\rho_0(1-\rho_0)<1/4$ otherwise the flux is not really limited, but
moreover we assume that $\delta$ is close to 0 for simplicity. This corresponds to a strong
limitation on the crossing at $x=0$. 

Using the control formulation, it is clear that $u(0,t)= -\delta t$ since the reward $-\delta$ for
$x=0$ is maximal---recall that we maximize the reward since $f$ is concave. Since $u$ is identified
on $x=0$, it only remains to solve a Dirichlet problem for the HJ-Equation in the domains $x>0$ and
$x<0$ separately. We can also guess that the solution is piecewise affine,
$$u(x,t)=at+bx \quad \hbox{with }a=-b(1-b)\;,$$
and by using this ansatz we solve the equations separately in different regions. Matching everything
is done by finding suitable lines originating from $(0,0)$ so that $u$ is globally continuous.

\medskip

\noindent\textbf{(a)} {If $x\neq0$ and $t$ is close to $0$} the flux limitation is not
interacting yet, so we use formula \eqref{traffic.formula} which yields here also
$$u(x,t)=\rho_0x-\rho_0(1-\rho_0)t\;.$$ 

\noindent\textbf{(b)} {If $t>0$ and $x$ close to $>0$} the flux limitation is acting on the
solution and using the boundary value $u(0,t)=-\delta$, we define $\rho_1$ and $\rho_2$ as the two
solutions of the equation $-\delta = -b(1-b)$, namely
$$ \rho_1=\frac{1+(1-4\delta)^{1/2}}2\quad\text{close to 1}\;,$$
$$ \rho_2=\frac{1-(1-4\delta)^{1/2}}2\quad\text{close to 0}\;.$$
In this region, the solution is given by $u(x,t)=\rho_1x-\delta t$ if $x<0$ and $\rho_2x-\delta t$
if $x>0$.

\noindent\textbf{(c)} {It remains to match everything} by continuity. For example, the continuity
condition 
$$ \rho_0x-\rho_0(1-\rho_0)t= \rho_1 x-\delta t \quad \hbox{implies  } 
x=\frac{\delta - \rho_0(1-\rho_0)}{\rho_1-\rho_0}t\; ,$$
where the coefficient of $t$ is strictly negative since the denominator is positive while for
$\delta$ small enough, $\rho_1>\rho_0$. This defines a first line $\Delta_1$, located on the left.
Similarly, $\Delta_2$ is defined by the matching condition using $\rho_0$ and $\rho_2$, here the
coefficient of $t$ is positive since for $\delta$ small enough, $\rho_2<\rho_0$.

Computing $\rho=u_x$, we find that $\rho(0^-,t)=\rho_1$ is close to $1$, which is reasonable since
the limited flux implies an accumulation of cars for $x<0$ close to $0$. On the contrary,
$\rho(0^+,t)=\rho_2$ is close to $0$ since the flux of cars is limited, therefore only a few cars go
through $x=0$. Figure~\ref{fig:traffic} gives a typical picture when $\rho_0>1/2$.

\begin{figure}[htp]
   \begin{center}
   \includegraphics[width=0.6\textwidth]{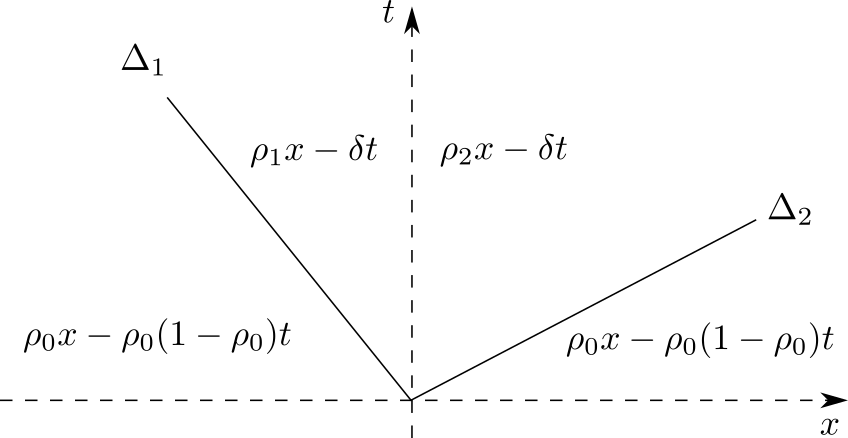}
   \caption{The solution $u$}
   \label{fig:traffic} 
   \end{center}
\end{figure}

\ 

We conclude this section by mentioning the case of moving constraints
$$f\big(\rho(y(t),t) \big)-\dot y(t) \rho(y(t),t) \leq g(t) \; ,$$
where $y$ and $g$ are given functions. Looking at the new function 
$$ w(x,t) = u(x+y(t),t)\; ,$$ 
we end up being in an analogous situation where again the theory of the previous chapter applies.
But, of course, we also use a completely formal argument to connect this problem with the HJ-one.

\chapter{Further Discussions and Open Problems}
\label{sec:SCQ}

\abstract{Here are collected, summarized and commented the results that are provided in the case of
codimension $1$ discontinuities in Part~\ref{part:codim1} and \ref{part:NA}. Some puzzling
questions are also considered.}

Let us first examine the three approaches we have described.

The first one, using {\em Ishii's notion of viscosity solutions}, has the advantage to be
very stable and universal in the sense that it can be formulated for any type of Hamiltonians, convex or not.
But Chapter~\ref{chap:Ishii} shows that it has poor uniqueness properties in the present situation. In the simple case
of the optimal control framework we have considered, with a discontinuity on an hyperplane $\H$ and with perhaps
a specific control on $\H$, we are able to identify the minimal solution ($\VFm$) and the maximal solution ($\VFp$):
if $\VFm$ is a natural value function providing the minimal cost over all possible controls, $\VFp$ completely ignores
some controls and in particular all the specific control on $\H$.

Why can $\VFp$ be an Ishii viscosity solution of the Bellman Equations anyway? The answer is that the Ishii 
subsolution condition on $\H$  is not strong enough in order to force the subsolutions to see all the particularities of 
the control problem on $\H$. 
This generates unwanted (or not?) subsolutions. We point out that, as all the proofs of Chapter~\ref{chap:Ishii} show, 
there is a  complete disymmetry between the sub and supersolutions properties in this control setting: this fact is 
natural and well-known due to the form of the problem but it is accentuated in the discontinuous framework.

This lack of uniqueness properties for Ishii viscosity solutions leads to consider different notions of solutions but, in some interesting applications, one may recover this uniqueness since $\VFm=\VFp$. We point out Lemma~\ref{lem:H1m.H2p.c} below which provides a condition under which $\HT=\HTreg$ and therefore $\VFm=\VFp$. This condition is formulated directly on the Hamiltonians and can sometimes be easy to check (see for example, Section~\ref{mg-KPP}).

In the {\em Network Approach}, one can either use the notion of {\em flux-limited solutions} or the notion 
of {\em junction viscosity solutions}. The first one is particularly well-adapted to control problems and has the great advantage to 
reinforce the  subsolutions conditions on $\H$ and, through the flux~limiter, to allow to consider various control problems at 
the same time by just varying this flux~limiter. The value functions $\VFm$ and $\VFp$ are reinterpreted in this framework as 
value functions associated to particular flux~limiters.

But we are very far from the universality of the definition of viscosity solutions since this ``max-max'' definition in the case of convex Hamiltonians has to be replaced by a ``min-min'' one in the case of concave ones, and it has no analogue for general ones. On the other hand, this notion of solution is less flexible in terms of stability properties compared to Ishii solutions.

The notion of {\em junction viscosity solution} tries to recover all the good properties of Ishii solutions for general Hamiltonians: it is valid for any kind of ``viscosity solutions compatible'' junction conditions, it is stable and the Lions-Souganidis proof (even if there are some limitations in Theorem~\ref{LS-CR}) is the only one which is valid for general Hamiltonians with Kirchhoff's boundary conditions. Though this approach is not as well-adapted to control problems as the flux~limiter one, it gives however a common formulation for problems when the controller wants to minimize some cost (which leads to convex Hamiltonians) or maximize it (which leads to concave Hamiltonians).

The Kirchhoff boundary condition is one of the most natural ``junction condition'' in the networks theory but a priori, it has no
connection with control problems. However, as it is shown by Proposition~\ref{kc-fl} together with Theorem~\ref{compBBC-IM}, this boundary condition is associated $\VFp$. The explanation is maybe in the next paragraph.

In fact, the main interest of the approach by junction solution, using the Lions-Souganidis comparison result,
is to provide the convergence of the vanishing viscosity method in the most general framework (with the limitations of Theorem~\ref{LS-CR}), 
without  using some convexity or quasi-convexity assumption on the Hamiltonians. In the convex 
setting, we have several proofs of the convergence to $\VFp$ which shows that it is the
most stable value function if we add a stochastic noise on the dynamic.

In the next parts, we examine stratified solutions in $\R^N$ or in general domains, $i.e.$ essentially the generalization of $\VFm$ which we aim at 
characterizing as the unique solution of a suitable problem with the right viscosity inequalities. And we will emphasize the (even more important) 
roles of the subsolution inequalities, normal controllability, tangential continuity...etc. But we will not consider questions related to $\VFp$ and the 
vanishing viscosity method, even if some of these questions are really puzzling.

This analysis generates a lot of questions.

The first one may concerns the limitations due to the assumptions of Theorem~\ref{LS-CR}: it is not completely clear that \TCs
is really necessary; maybe a different proof, avoiding the tangential regularization, can handle general Hamiltonians without this
superfluous hypothesis. 

All the other questions concern the extensions to higher codimension discontinuities of the notions of flux-limited and junction viscosity
solutions. Clearly the first step should be to have the right space of test-functions (like $\PC1$ above). It is not very difficult to guess
what this space could be: in the stratified case, i.e. if the discontinuities for a stratification $\M=(\Man{k})_k$, one may perhaps use
continuous functions those restrictions to each $\Man{k}$ are $C^1$ and with derivatives which have continuous extensions to $\overline{\Man{k}}$. To write that is one point, to make a concrete proof is an other one.

We point out anyway how Lemma~\ref{lem:comp.fundamental} is closely related to flux-limited solutions by looking
at trajectories which either leave $\mathcal{M}$ (suggesting a $H_1^+-H_2^-$-type inequality) or stay on $\mathcal{M}$.
It seems clear that a pde analogue of this lemma should exists  and allow to obtain a comparison
result for---at least---HJB-equations 
by a pure pde method, even in the stratified case of Part~\ref{stratRN}.

Finally we come back to the question which is clearly the most puzzling for us: one of the main result of this part is the convergence of 
the vanishing viscosity method to $\VFp$, the maximal viscosity subsolution in the control framework. What is the analogue of this
result in the case of higher codimension discontinuities? Is the convergence to the maximal viscosity subsolution always true? And of course,
can we identify this maximal viscosity subsolution in the control framework via an explicit formula?\index{Vanishing viscosity method!general questions 
on}


\part{General Discontinuities: Stratified Problems}
\label{stratRN}
\fancyhead[CO]{HJ-Equations with Discontinuities: Stratified Problems}


\chapter{Stratified Solutions}
\label{chap:strat-def}

\abstract{This chapter is devoted to present the notions of stratified solutions and to provide the
associated comparison result. Compared to \cite{AEYW}, ``weak'' and ``strong'' stratified solutions
are introduced and it is proved that the comparison holds for regular, weak solutions. Which implies
that comparison also holds for strong solutions, since a regular weak solution is a strong solution.}

\section{Introduction}
\label{sect:intro.strat}

Throughout Part~\ref{stratRN}, we consider Hamilton-Jacobi-Bellman Equations with more general
discontinuities than hyperplanes. Those discontinuities can be of any codimension but with the
restriction that they form a ``Whitney stratification'' and even a \TFS, \cf Section~\ref{sect:whitney}. 

This generality is at the expense of considering only equations which are closely related to control
problems (hence with convex Hamiltonians) but with the advantage that we do not have to deal with
existence results: as can be expected, the value function of the associated control problem is a
solution, even if this fact will not be completely obvious, \cf Chapter~\ref{chap:stratcontr}.
\index{Whitney stratification}

We always assume that we are in the ``good framework for discontinuities'': even if some of these
assumptions can certainly be weakened, this general framework seems the most natural for us since,
as we have already pointed out several times, the basic hypothesis we impose are useful---if not
unavoidable---in the proof of {\em any} results. 

This chapter is devoted to introduce the notion of \emph{Stratified Solutions} in this framework
and to present a comparison result which is valid under ``natural'' assumptions. We also show that
the stratified solution corresponds to the minimal Ishii supersolution (see Section~\ref{sec:Strat-Is}). 

We give here two notions of stratified solution: a \emph{weak} one and a \emph{strong}
one, the strong notion involving additional inequalities with respect to the weaker one. The difference
between these notions can be understood in a better way after reading Section~\ref{sect:natural.Ishii}:
for the strong one, we impose the $\F_*\leq 0$-inequality on the discontinuities while, for the weak one,
we just impose ``tangential inequalities''. Each notion may have a specific interest, in particular for stability
results but also for further developments, but they turn out to be the same in the ``good framework''
which we always use, see \HBASF below (see Section~\ref{rweqs}). 

More concretely, we are given a general HJB Equation of the form
\begin{equation}\label{eq:super.H.strat}
    \F(x,t,U,DU)=0\quad\text{in}\quad\R^N\times[0,\Tf]\;,
\end{equation}
where $DU=(D_x U, D_t U)$ and 
\begin{equation}\label{hamil.strat}
    \F(x,t,r,p):=\sup_{(b,c,l)\in\BCL(x,t)}\big\{ -b\cdot p + c r- l \big\}\;,
\end{equation}
where $\BCL: \R^N\times[0,\Tf] \to \R^{N+3}$ is a set-valued map (\cf Section~\ref{sect:disc.pb}).
And we define the initial Hamiltonian as
\begin{equation}\label{eq:Finit.start}\F_{init}(x,r,p_x)=
    \sup_{\substack{((b^x,0),c,l)\in\BCL(x,0)}}\big\{-b^x \cdot p_x + c r - l \big\}\;.
\end{equation}

The fundamental assumptions we make in this part are the following\label{page:HBASF}

\begin{assumption}{\HBASF}{Basic Assumptions on the Stratified Framework.}\\[-1.3cm]
\begin{enumerate}
    \item[$(i)$] There exists a \TFS $\M=(\Man{k})_{k=0...(N+1)}$ of $\R^{N}\times
        (0,\Tf)$ such that, for any $r\in \R,p\in \R^{N+1}$, $(x,t)\mapsto\F(x,t,r,p)$ is continuous on $\Man{N+1}$ and may be
        discontinuous on $\Man{0} \cup \Man{1}\cup \cdots \cup \Man{N}$. Moreover $(0_{\R^N},1)
        \notin (T_{(x,t)}\Man{k})^\bot$ for any $(x,t) \in \Man{k}$ and for any
        $k=1...N$~\footnote{This assumption, whose aim is to avoid ``flat part'' of $\Man{k}$ in
        time, will be redundant to the normal controllability assumption in $\R^N\times(0,\Tf)$.  }.
        In the same way, there exists a \TFS $\M_0=(\Man{k}_0)_{k=0...N}$ of $\R^N$ such
        that, for any $r\in \R,p_x\in \R^{N}$, the Hamiltonian $x\mapsto\F_{init}(x,r,p_x)$ is continuous on $\Man{N}_0$ and may be
        discontinuous on $\Man{0}_0 \cup \Man{1}_0 \cup \cdots \cup
        \Man{N-1}_0$.
    \item[$(ii)$] The ``good framework for HJB Equations with discontinuities'' holds for
        Equation~\eqref{eq:super.H.strat} in $\OO=\R^N\times (0,\Tf)$ associated to the
        stratification $\M$.
    \item[$(iii)$] The ``good framework for HJB Equations with discontinuities'' holds for the
        equation $\F_{init}=0$ in $\OO=\R^N$, associated to the stratification $\M_0$. 
\end{enumerate}
\end{assumption}

\index{Good framework for HJE!for stratified problems}We recall that the assumptions for a ``Good
Framework for HJ Equations with Discontinuities'' are that \HBCL, \TCBCL and \NCBCL hold. We refer
to Section~\ref{sec:GFHJD} where the connections with Hamiltonian assumptions \Mong, \TC, \NCe are
described.

\label{rem:MM0}The reader may be surprised that, in \HBASF, the stratifications $\M$ and $\M_0$ are defined
independently the one to the other; in particular, it may seem natural that $\Man{k}_0$ contains at
least the trace of $\Man{k+1}$ at $t=0$. 

The simple framework of Chapter~\ref{chap:BasicFram} is helpful to explain why this is not the case:
indeed, if we consider the case when $(b(x,t,\alpha), c(x,t,\alpha), l(x,t,\alpha))$ is
discontinuous in $x,t$ on a stratification $\M$ for $t>0$, so is $\F(x,t,r,p)=p_t+H(x,t,r,p_x)$ but,
for $t=0$, the $\F_{init}$-equation, which is just $u(x,0)=u_0(x)$ in $\R^N$, does not present any
discontinuity if $u_0$ is continuous.

Hence the discontinuity in $\BCL$ creates difficulties for $t>0$ but not for $t=0$. In the same way, we can consider
the case when $b(x,t,\alpha), c(x,t,\alpha), l(x,t,\alpha)$ are continuous---hence $\Man{N+1}=\R^N
\times (0,\Tf)$---and $u(x,0)$ is obtained by solving a stationary stratified problem in $\R^N$, in
which case the difficulty is now at $t=0$ but not for $t>0$. 

These two examples show that actually the stratifications $\M$ and $\M_0$ are independent the one to
the other, with difficulties to solve the equations which are also independent.

\section{Definition of weak and strong stratified solutions}
\label{sec:strat-sols}

In order to state a definition, we introduce Hamiltonians $\F^k$, defined as follows:
if $(x,t)\in \Man{k}$, $r \in \R$ and $p \in T_{(x,t)}\Man{k}$,
we set
\begin{equation}\label{def:Fk}
\F^k(x,t,r,p):=\sup_{\substack{(b,c,l)\in\BCL(x,t)\\ b\in T_{(x,t)}\Man{k}}}\big\{ -
    b\cdot p +cr - l\big\}\;.
\end{equation}
Similarly, for $t=0$ we define
\begin{equation}\label{def:Fk-init}
\F^k_{init}(x,r,p_x):=
    \sup_{\substack{((b^x,0),c,l)\in\BCL(x,0)\\ b^x\in T_{x}\Man{k}_0}}\big\{-b^x \cdot p_x + c r - l
    \big\}.  
\end{equation}
We may also use these definitions for $p\in \R^{N+1}$ or $p_x \in \R^N$ since it is clear that there is no
contribution from the $(T_{(x,t)}\Man{k})^\bot$ or $(T_{(x,t)}\Man{k}_0)^\bot$ part of $p$ or $p_x$.

In the framework of Part~\ref{part:codim1}, as the reader may guess, we have $\Man{N} =
\H\times (0,\Tf)$, $\Man{N+1} = (\Omega_1\cup \Omega_2)\times (0,\Tf)$ and $\HT$ is exactly $\F^{N}$,
while $\Man{0},\cdots \Man{N-1}=\emptyset$ . 

In the sequel, the notation \HJBS refers to problem \eqref{eq:super.H.strat}, seen in the context
of stratified solutions, that we detail below. 
The notion of stratified supersolution, denoted by \SSup, is nothing but the usual Ishii
supersolution definition involving $\F^*=\F$. On the other hand, we introduce two notions of weak
and strong subsolutions, respectively denoted by \wSSub and \sSSub, recalling that, because of \HBASF, the Hamiltonians $ \F^k, \F_{init}^k$ are continuous for all $k$.  
\index{Stratified solutions!definition}

\begin{definition}\label{def:HJBSD}\label{defiStrat}
   \emph{--- Stratified sub and supersolutions of \HJBS.}\smsp
    {\bf 1.\,---} \SSup {\rm :} A locally bounded function $v : \R^N\times [0,\Tf[ \to \R$ is a
    stratified supersolution of \HJBS if $v$---or equivalently $v_*$---is an Ishii supersolution of 
    \eqref{eq:super.H.strat}. 
    
    \medskip

    \noindent {\bf 2.\,---} \wSSub {\rm :}
    A locally bounded function $u : \R^N\times [0,\Tf[ \to \R$ is a {\em weak} stratified subsolution
    of \HJBS if 
    \begin{enumerate}
    \item[$(a)$] for any $k=0,...,(N+1)$, $u^*$ is a viscosity subsolution of 
        $$ \F^k\big(x,t,u^*,Du^*\big) \leq 0 \quad \hbox{on  }\Man{k},
        $$
     \item[$(b)$] similarly, for $t=0$, and $k=0..N$, $u^*(x,0)$ is a viscosity subsolution of
            $$ \F_{init}^k(x,u^*(x,0),D_x u^*(x,0)) \leq 0 \quad \hbox{on  }\Man{k}_0\;.$$
    \end{enumerate}

    \noindent {\bf 3.\,---} \sSSub {\rm :} 
    A locally bounded function $u : \R^N\times [0,\Tf[ \to \R$ is a {\em strong}
    stratified subsolution of \HJBS if it is a \wSSub and satisfies additionally
    \begin{enumerate}
        \item[$(a)$] $\F_*\big(x,t,u^*,Du^*\big)\leq0\quad \hbox{in }\R^N\times(0,\Tf$)\;,
        \item[$(b)$] $(\F_{init})_*(x,u^*(x,0),D_x u^*(x,0)) \leq 0 \quad \hbox{in  }\R^N\;.$
    \end{enumerate}

    \noindent {\bf 4.\,---}  A weak or strong \emph{stratified solution} is a
function which is both a \SSup and either a \wSSub or a \sSSub.
\end{definition}

As usual, we will say that $u$ is an $\eta$-strict (weak or strong) stratified
subsolution if the various subsolution inequalities of the type ``$\mathbb{G}\leq0$'' are replaced
by a ``$\mathbb{G}\leq -\eta$'' inequality, the constant $\eta>0$ being independent of~$(x,t)$.

The difference between weak and strong stratified solutions can be better understood through the
discussion in Section~\ref{sect:natural.Ishii}, let us comment  on this now.

The notion of ``strong'' stratified (sub)solution which was used in \cite{AEYW} is the easiest to
interpret and maybe the more natural one from the viscosity solutions---or pde---point of view:
Part~\ref{part:codim1} teaches us that a subsolution inequality is missing on the discontinuity $\H
\times (0,\Tf)$ in order to get a uniqueness property, and that adding the right one solves this
problem. Therefore, it is not surprising to introduce the concept of ``stratified solution'' by
super-imposing {\em additional} subsolutions inequalities on each set of discontinuity $\Man{k}$,
including those at time $t=0$. We point out that these additional subsolution conditions are real
``$\Man{k}$-inequalities'', \ie they are obtained by looking at maximum points of $u^*-\varphi$ on $\Man{k}$
where $\varphi$ is a test-function which is smooth on $\Man{k}$.

Of course, at first glance, removing the $\F_*\leq 0$-inequality does
not seem to go in the right direction since some subsolution property is clearly missing. This is why, in
\cite{AEYW}, looking for ``additional'' (and not different) subsolution properties seems more natural; and
Section~\ref{rweqs} gives an other a posteriori justification of the interest of ``strong'' stratified subsolutions.

But, on the other hand, from the control point of view, the
$\F_*$-inequality (as well as the $(\F_{init})_*$-one) is not so natural for reasons explained in
Section~\ref{sect:natural.Ishii}. In that sense, a controller may understand in a better way the notion of ``weak''
stratified (sub)solution since all the inequalities have a clear sense in terms of control.

However, the defect of the notion of ``weak'' stratified (sub)solution is that it completely decouples the
subsolution inequalities on the different $\Man{k}$ and by doing so, opens the possiblity to
generate ``artificial'' subsolutions with uncorrelated values on these discontinuities.

The reconciliation of these two points of view involves in a central way the question of the regularity
of subsolutions on the various $\Man{k}$. Spoiling the results of Section~\ref{rweqs}, we can
summarize the answer as
$$\text{``\ if \HBASF holds,  then regular \wSSub $=$ \sSSub.\ ''}$$

Roughly speaking, this means that under Assumption \HBASF, the $\F_*$ and $(\F_{init})_*$
inequalities are just used to obtain the regularity of subsolutions and, once this regularity is
obtained, only the ``weak'' notion of stratified subsolution plays a role. Actually, as far as
the regularity issue is concerned, $\F_*$ and $(\F_{init})_*$ are not playing any special role in the
definition of \sSSub; in fact they could be replaced by other Hamiltonians satisfying \NCe which is
the real key point in order to get regularity. 
This remark may seem anecdotical here but it will play a real role in
Part~\ref{S-BC}. 

\

In the rest of the book, since we will always use Assumption \HBASF, we will always be in a context
where ``weak'' and ``strong'' stratified solutions coincide and we will use mainly \emph{strong
stratified solutions} (or we will mention the difference if there is one). 
Therefore the terminologies ``sub/supersolution of \HJBS'' or ``stratified
sub/supersolution of \eqref{eq:super.H.strat}'' refer to the above definition combining \sSSub
and \SSup. We also sometimes use the terminology \emph{Standard Stratified Problem} \SSP referring
to a problem in the form of \eqref{eq:super.H.strat}--\eqref{hamil.strat}--\eqref{eq:Finit.start}
satisfying \HBASF, understanding sub and supersolutions in the strong sense of
Definition~\ref{def:HJBSD}.

The next section is devoted to show that \sSSub are regular, being more precise about the term
``regular''itself. Then we will show that, under Assumption \HBASF, a comparison holds between
regular \wSSub and \SSup, which will allow us to prove also that regular \wSSub and \sSSub are the
same.

\section{The regularity of strong stratified subsolutions}
\label{sec:regstratsub}

We recall that the regularity of discontinuous functions is defined in Definition~\ref{def:regular};
based on it, we define the regularity of stratified subsolutions (weak or strong).
\begin{definition}\emph{--- Regularity of stratified subsolutions.}\smsp
    Let $u:\R^N\times [0,\Tf] \to \R$ be a \usc weak or strong stratified subsolution of
\eqref{eq:super.H.strat}. We say that $u$ is a regular subsolution if
    \begin{enumerate}
        \item[$(i)$] for any $k<N+1$, $u$ is $\omega^k$-regular on $\Man{k}$, where
            $$\omega^k=\Man{k+1}\cup \Man{k+2}\cup\cdots\cup \Man{N+1}\;.$$
            In other words, for any $(x,t) \in \Man{k}$ and $k<N+1$, 
            \begin{equation}\label{bonnelim11}
            u(x,t) = \limsup\{u(y,s),\ (y,s)\to (x,t),\ (y,s) \in \omega^k\}\;.
            \end{equation}
            Moreover, for the special case where $(x,t) \in \Man{N}$, we also have
            \begin{align}\label{bonnelim12}
            u(x,t) = &\limsup\{u(y,s),\ (y,s)\to (x,t),\ (y,s)\in M^{(x,t)}_+ \}\nonumber \\
            =& \limsup\{u(y,s),\ (y,s)\to (x,t),\ (y,s)\in M^{(x,t)}_- \},
            \end{align}
            where, for $r$ small enough,
            $M^{(x,t)}_+,M^{(x,t)}_-\subset\Man{N+1}\cap B((x,t),r)$ are the locally disjoint
            connected components of $\big((\R^N\times (0,\Tf))\setminus\Man{N}\big)\cap
            B((x,t),r)$.
        \item[$(ii)$] For $t=0$, for any $k<N$, the function $x\mapsto u(x,0)$ is
            $\omega_0^k$-regular on $\Man{k}_0$ where 
            $$\omega_0^k=\Man{k+1}_0\cup \Man{k+2}_0\cup\cdots\cup \Man{N}_0$$
            and, in the special case where $(x,0) \in\Man{N-1}_0$, we also have
            \begin{align}\label{bonnelim12-0}
            u(x,0) = &\limsup\{u(y,0),\ y\to x,\ (y,0)\in M_{0,+}^{(x,t)} \}\nonumber \\
            =& \limsup\{u(y,s),\ y\to x,\ (y,0)\in M_{0,-}^{(x,t)} \},
            \end{align}
            where, for $r>0$ small enough,
            $M_{0,+}^{(x,t)},M_{0,-}^{(x,t)}\subset\Man{N}_0\cap (B(x,r)\times\{0\})$ 
            are the locally disjoint connected components of
            $\big((\R^N\times\{0\})\setminus\Man{N-1}_0\big)\cap  (B(x,r)\times\{0\})$.
    \end{enumerate}
\end{definition}

The result on strong stratified subsolutions is the following. 
\begin{proposition}\label{prop:reg-sub-strat}Assume that \HBASF holds. 
    Then any strong stratified subsolution is regular.
\end{proposition}

We leave the proof of this important proposition to the reader since it is a routine application of
Proposition~\ref{reg-sub} after a suitable flattening of the $\Man{k}$ we are interested in, using
the definition of a \TFS.

    Again we insist on the very anecdotical role played by the Hamiltonians $\F_*, (\F_{init})_*$ and
    the $\F_*, (\F_{init})_*\leq 0$ inequalities in this result: only Assumption \NCBCL is
    playing a key role. As a consequence, if we replace the $\F_*, (\F_{init})_*\leq 0$ inequalities 
    in the definition of \sSSub, by some other ones like $\G_*, (\G_{init})_*\leq 0$ with
    $\G_*, (\G_{init})_*$ satisfying \NCBCL, we would still get regular subsolutions. This remark is
    important when dealing with stability results: given a sequence $(\ue)_\e$ of \sSSub for
    Hamiltonians $(\F_\e)_\e$, then $\limssup \ue$ is a subsolution for $\underline{\F}:=\limiinf
    \F_\e$. With suitable assumptions, $\underline{\F}$ may satisfy \NCBCL, while, in general, it is not
    clear that $\underline{\F}=(\overline{\F})_*$, where $\overline{\F}=\limssup \F_\e$. Hence a stability result
    for ``strong'' solutions is far more delicate to obtain than for ``weak'' solution.

    This remark yields another justification for introducing the notion of \wSSub, apart from being
    natural from the control viewpoint: while their needed regularity may come indeed from $\F_*,
    (\F_{init})_*$ inequalities---referring to strong solutions---it may also come from other
    inequalities coming for Hamiltonians satisfying \NCBCL. But also, it can derive from a
    particular situation where \NCBCL may not even be satisfied. We refer the reader to Section~\ref{abl} where the connections
    between the regularity of subsolutions and the existence of certain viscosity inequalities on the
    boundary are discussed for state-constrained problems; such arguments can also be applied on $\Man{k}$ for standard stratified
    problems.

    As an example, the reader may consider the cases when all the dynamics are pointing toward
    $\Man{k}$ for some $k$: clearly \NCBCL is not satisfied but the subsolutions are expected to be
    regular on $\Man{k}$. Under suitable assumptions, it should be possible to handle such
    cases by solving a problem on $\Man{k}$, and then by using the solution on $\Man{k}$ as a
    Dirichlet data for the problem in $\R^N \times (0,\Tf)$.

Despite we are not going to use it in this part, let us mention an immediate consequence of
Proposition~\ref{prop:reg-sub-strat} which will be useful in the case of state-constrained problems.
\begin{corollary}\label{cor:reg-MN}
    Assume that \HBASF holds and let $u:\R^N\times [0,\Tf] \to \R$ be an \usc strong stratified
    subsolution of \eqref{eq:super.H.strat}. Then, for any $k<N+1$, $u$ is $\Man{N+1}$-regular on
    $\Man{k}$. Similarly for $t=0$, for any $k<N$, the function $x\mapsto u(x,0)$ is
    $\Man{N}_0$-regular on $\Man{k}_0$.
\end{corollary}

\begin{proof}
    We just sketch it in the case of $\Man{k}$ (\ie for $t>0$) because it is an easy consequence of
    Proposition~\ref{prop:reg-sub-strat} by backward induction. The case $t=0$ is of course similar.
    Notice first that the result clearly holds for $k=N$ as a direct application of
    Proposition~\ref{prop:reg-sub-strat}. 

    Now, assume that the result holds for $k=N, \cdots,(N-l)$ and take $(x,t) \in \Man{N-l-1}$. By
    Proposition~\ref{prop:reg-sub-strat}, 
    $u(x,t)=\limsup_\e u(\xe,\te)$ with $(\xe,\te)\in \Man{N-l}\cup\cdots\cup \Man{N+1}$. 

    But we can use the $\Man{N+1}$-regularity at $(\xe,\te)\in \Man{N-l}\cup\cdots\cup \Man{N+1}$ to
    build a new sequence $(\xe',\te')\in \Man{N+1}$ such that $u(x,t) =\limsup_\e u(\xe',\te')$,
    implying that the induction works since the regularity result holds true for $k=N-l-1$.   
\end{proof}

\section{The comparison result}
\label{sec:compstrat}

The main advantage of the concept of (weak and strong) stratified solutions is reflected in the
comparison principle which we state now. \index{Comparison result!for stratified solutions
($\R^N$-case)}

\begin{theorem}\emph{--- Comparison result for stratified solutions.}\label{comp-strat-RN} \smsp
    $(i)$ Assuming \HBASF, a comparison result holds between bounded regular \wSSub and bounded \SSup of
    Equation~\eqref{eq:super.H.strat}.\\[2mm]
    $(ii)$ Assuming \HBASF, a comparison result holds between bounded \sSSub and bounded \SSup of
    Equation~\eqref{eq:super.H.strat}.
\end{theorem}

\begin{proof} 
    Of course, the second part of the result is an immediate consequence of the first one
    because of Proposition~\ref{prop:reg-sub-strat}.

Now we turn to the proof of the first part. Essentially the proof follows the main steps as the proof of 
Theorem~\ref{thm:minimal.charac} where it is shown that $\VFm$ is the unique solution of the Bellman Equation
with the $\HT$-complemented inequality, which turns out to be an $\Man{N}$-inequality in the stratified setting. The
only difference is that we have to use the more sophisticated form of Theorem~\ref{thm:sub.Qk.dpp}. 

    Before describing these main steps, let us introduce some notations and perform some reductions.
    Let $u,v : \R^N\times [0,\Tf[ \to \R$ be respectively a bounded \usc regular stratified subsolution and
    a bounded \lsc stratified supersolution of Equation~\eqref{eq:super.H.strat}. Our aim is to
    show that $u \leq v$ in $\R^N\times [0,\Tf[$\footnote{The reason why we do not include $\Tf$ in
    the comparison will be clarified later on.}.

    This inequality is proved via two successive comparison results: first, one has to show that
    $u(x,0) \leq v(x,0)$ in $\R^N$ which derives from a comparison result associated to the
    stationary equation $\F_{init}=0$. Then, to prove that $u \leq v$ in
    $\R^N\times ]0,\Tf[$, using a comparison for the evolution problem.
    The global strategy to obtain the comparison is the same in both cases and the changes to
    pass from one to the other are minor. Therefore we are going to provide the full proof only in
    the evolution case, admitting that $u(x,0) \leq v(x,0)$ in $\R^N$.

    \medskip

    \noindent\textbf{Reductions ---}
    In order to prove these comparison results, we perform the following changes which are based on
    Assumption \HBCLb.  We first use the by-now classical change $$ \bar u(x,t) =
    \exp(-Kt)u(x,t)\quad\hbox{ and } \quad \bar v(x,t) = \exp(-Kt)v(x,t)\;,$$
    which, according to \HBCLb-$(ii)$, allows to reduce to the case when $c\geq 0$ for any
    $(b,c,l)\in \BCL(x,t)$ and $(x,t)\in \R^N\times [0,\Tf]$. 
    We may also assume that $c\geq 1$ if
    $-b^t\geq \uc$ given by \HBCLb-$(iv)$.
    Notice  that Assumption \HBCLb-$(ii)$ implies that
    $c\geq 0$ if $t=0$ and $b^t=0$, hence the Hamiltonian $\F_{init}(x,r,p_x)$ is increasing in $r$.

    Next, adding $C_1 t +C_2$ to $\bar u$ and $\bar v$ and using \HBCLb-$(iv)$, we can assume
    without loss of generality that $l \geq \uc$ for any element $(b,c,l) \in \BCL(x,t)$,
    for any $(x,t)\in \R^N\times [0,\Tf]$.

    In the comparison proofs, both for $t\in (0,\Tf)$ and $t=0$, we use in a key way
    that $c\geq 0$ and $l \geq \uc$ for any element $(b,c,l) \in \BCL(x,t)$,  any $(x,t)\in
    \R^N\times [0,\Tf]$. Indeed, these properties together with the convexity of $\F$ and $\F_{init}$
    allow us to reduce to the case of {\em strict subsolutions}, a favorable situation both in the
    stationary and evolution case.

    Then, the comparison proof in $\R^N\times ]0,\Tf[$ is done in five steps.

\medskip

\noindent\textbf{Step 1:} \emph{Reduction to a local comparison result \LCR-evol} --
    To do so, we adapt in a suitable way the ideas introduced in Section~\ref{sect:htc}. The precise
    result is the 
\begin{lemma}\label{lem:loc-strat} Let \HBCL hold and $\psi_\mu : \R^N \times [0,\Tf]\to \R$ 
    be defined by
    $$ \psi_\mu(x,t):=-\mu (1+|x|^2)^{1/2}\;.$$
    There exists $\mus:=\mu(M,\uc)>0$ such that for any $(x,t) \in \R^N \times [0,\Tf]$ and 
    $(b,c,l)\in \BCL(x,t)$, 
    $$ -b\cdot (D_x \psi_\mus (x,t), D_t \psi_\mus (x,t))+c\psi_\mus(x,t)
     -l\leq \mus M + 0 - \uc < -\uc/2\;.$$
    In particular, $\psi_\mus$ is a $(\uc/2)$-strict \sSSub and it follows that
    \begin{enumerate}
        \item[$(i)$] for $\alpha\in(0,1)$, \LOCaEV is satisfied by
        $$ \bar u_\alpha (x,t):= \alpha \bar u(x,t)+(1-\alpha)\psi_\mus(x,t)\;;$$
        \item[$(ii)$] \LOCbEV is satisfied by considering 
        $$\bar u_\alpha^\delta(x,t):=\bar  u_\alpha (x,t)- \delta (|x-\xb|^2+|t-\tb|^2)$$
        where $(\xb,\tb)$ is the point where we wish to check \LOCbEV and $\delta>0$ is small enough.  
    \end{enumerate}
\end{lemma}

We leave the easy proof of this (important) result to the reader since it presents no difficulty at
all. We point out anyway that the checking of \LOCaEV uses techniques related to what is called the
``convex case'' in Section~\ref{sect:htc}, while the checking of \LOCbEV relies on the simplest
argument presented in the second particular case of this section. This strategy, based on
Assumption \HBCLb$-(iv)$, allows to overcome the difficulty mentioned in
Remark~\ref{rem:htc}-$(iii)$, appearing in equations like $\max(\G(x,u,D_xu);|D_x u|-1)=0$.

Since all the above reductions do not affect the regularity of the subsolution, we are reduced to
prove local comparison results between regular subsolutions and supersolutions. For the sake of
simplicity of notations, from now on we just denote by $u$ a \emph{strict regular stratified
subsolution} and $v$ a stratified supersolution.

    \medskip

    \noindent\textbf{Step 2:} \emph{Local comparison and argument by induction} -- In order to
    prove \LCR-evol we argue by induction. But using Theorem~\ref{thm:sub.Qk.dpp} we have to show,
    at the same time a local comparison result not only for Equation~\eqref{eq:super.H.strat} but
    also for equations of the type $\max(\F(x,t,w,Dw), w-\psi)=0$ where $\psi$ is a continuous function.
    In fact, with the assumptions we use, there is no difference when proving \LCR-evol for these
    two slightly different equations but, in order to be rigorous, we have to consider the
    ``obstacle'' one, which reduces to the $\F$--one if we choose $\psi(x)=K$ where the constant $K$
    is larger than $\max(||u||_\infty,||v||_\infty)$.

    For the sake of simplicity, we use below the generic expression $\psi$--Equation for the
    equation $\max(\F(x,t,w,Dw), w-\psi)=0$ and we will always assume that $\psi$ is a continuous
    function, at least in a neighborhood of the domain we consider.

    We are then reduced now to show that, for any $(\xb,\tb) \in \R^N \times (0,\Tf)$:\\

    \label{page:LCRxt}
    \noindent\emph{\LCRxt {\rm :} There exists $r=r(\xb,\tb)>0$ and $h=h(\xb,\tb)\in(0,\tb)$ such
    that, if $u$ and $v$ are respectively a strict regular stratified subsolution\footnote{According to the
    type of obstacle $\psi$ we have to use in the proof of Theorem~\ref{thm:sub.Qk.dpp}, we can
    assume \wlg  that $u \leq \psi -\delta$ for some $\delta>0$ in $\overline{Q^{\xb,\tb}_{r,h}}$
    and therefore a strict subsolution of $\F=0$ or of the $\psi$-Equation have essentially the same meaning.}
    and a stratified supersolution of some $\psi$--Equation in $Q^{\xb,\tb}_{r,h}$ and if
    $\displaystyle \max_{\overline{Q^{\xb,\tb}_{r,h}}}(u-v) >0$, then
    $$ \max_{\overline{Q^{\xb,\tb}_{r,h}}}(u-v) \leq \max_{\partial_p Q^{\xb,\tb}_{r,h}}(u-v)\;,$$
    where we recall that $\partial_p Q^{\xb,\tb}_{r,h}$ stands for the parabolic boundary of
    $Q^{\xb,\tb}_{r,h}$, namely here $ \partial {B(\xb,r)}\times [\tb-h,\tb]\cup
    \overline{B(\xb,r)}\times \{\tb-h\}$.    }\\

    It is clear that \LCRxt holds in $\Man{N+1}$ since $\F^{N+1}$ and all the $\psi$--Equations
    satisfy all the property ensuring a standard comparison result in the open set $\Man{N+1}$;
    therefore \LCRxt is satisfied for $r$ and $h$ small enough---see Section~\ref{simple-ex-comp}.

    In order that it holds for $(\xb,\tb)$ in any $\Man{k}$, we use a (backward) induction on $k$
    and more precisely, we introduce the property\\

    \noindent\PP{k}:=$\Big\{ $ \emph{\LCRxt holds for any $(\xb,\tb) \in \Man{k}\cup
    \Man{k+1}\cup\cdots \cup \Man{N+1}\Big\} $}\;.\\

    Since \PP{N+1} is true, the core of the proof consists in showing that \PP{k+1} implies \PP{k}
    for $0\leq k \leq N$.  To do so, we assume that $(\xb,\tb) \in \Man{k}$ and want to prove that
    \LCRxt holds provided \PP{k+1} is satisfied.

    \medskip

    \noindent\textbf{Step 3:} \emph{Regularization of the subsolution}
    \index{Regularization of subsolutions}
    -- In order to apply the ideas of Section~\ref{sec:regplus-subsol}, we use the
    definition of an \TFS which allows us to assume that $\xb = 0$, $\tb>0$ and
    that we are in the case when $\Man{k}$ is a $k$-dimensional affine space parametrized by
    $(t,x_1,\cdots,x_{k-1})$, given by the equations $x_{k}=x_{k+2}=,\cdots, =x_N =0$. This
    reduction is based on a $C^{1,1}$-change of variable in $x$ which is done {\em only for the
    regularization step} and then we come back to the initial framework by the inverse of the
    change.

    In the new setting, we keep the notations $\F$, $\F^j$ (for all $j$) and $u$. We just point out
    here that the $t$-variable is always part of the tangent variables which explains some
    restriction in the assumption concerning the behavior of $\F^l$ in $t$, \cf\ \TC. Before
    proceeding, we emphasize the fact that, since $r$ and $h$ may depend on $(\xb,\tb)$, we can
    handle without any difficulty the localization to reduce to the case of a tangentially flat stratification.

    \index{Regularization of subsolutions} 
    Since $(\xb,\tb)=(0,\tb)\in \Man{k}$, we may assume that $Q^{x,t}_{r,h}$ only contains points of
    $\Man{k}, \Man{k+1},\cdots,\Man{N+1}$ and, by assumption\footnote{or by using using
    Proposition~\ref{prop:reg-sub-strat} for the reader who is just interested in strong stratified
    subsolutions.}, we know that the subsolution $u$ is $\omega^k$-regular on
    $\Man{k}$ where $\omega^k=\Man{k+1}\cup \Man{k+2}\cup \cdots\cup\Man{N+1}$.

    In order to regularize the subsolution and apply Proposition~\ref{reg-by-sc}, we 
    make the change of functions
    $$ \tilde u(x,t) = -\exp(-\alpha u(x,t))\quad \hbox{and} \quad 
    \tilde v(x,t) = -\exp(-\alpha v(x,t))\;.$$
    Indeed, a priori the initial $\F$ and $\F^j$ do not satisfy \Mong while the new Hamiltonians
    obtained after this exponential change satisfy \Monu if $\alpha$ is small enough. However, these
    new Hamiltonians are not necessarily convex in $r$ and $p$.

    Hence we first apply Proposition~\ref{reg-by-sc} to regularize the strict subsolution in
    $Q^{x,t}_{r,h}$, using the variables $y=(t,x_1,\cdots,x_{k-1})$, $z=(x_{k},x_{k+2},\cdots,x_N)$
    and the $ G((y,z),u,p)$ corresponding to $\max(\F_*(x,t,r,p),\F^j (x,t,r,p),u-\psi)$ but with
    the new Hamiltonian obtained with the above change of variable.  This approximation by a
    sup-convolution in the tangential variables leads to a Lipschitz continuous subsolution which is
    semi-convex in the tangential variables $y$.

    To proceed in order to obtain a sequence of strict stratified subsolutions which are $C^1$ in
    the variables $y=(t,x_1,\cdots,x_{k-1})$, there are two options: either we use
    Proposition~\ref{C1-reg-by-sc} with Remark~\ref{rem:r-nonconvex} since the new Hamiltonians
    satisfy \Monu but are not necessarily convex in $r$ or we make the change back and we use
    Lemma~\ref{combconvsub} to avoid assumption \Monu.

    In any case, applying back the change of variables if necessary, and using that the above
    procedure gives a strict stratified subsolution in a neighborhood of $(\xb,\tb)=(0,\tb)$, we
    find that  there exists $r,h>0$, $t'>\tb$ and a sequence $(\ue)_\eps$ of subsolutions of the
    stratified problem in $Q^{\xb,t'}_{r,h}$, which are in $C^0\left(\overline{Q^{\xb,t'}_{r,h}}\right)\cap
    C^1\left(\Man{k}\cap \overline{Q^{\xb,t'}_{r,h}}\right)$ and are all $(\eta/2)$-strict subsolutions of
    Equation~(\ref{eq:super.H}) in $Q^{\xb,t'}_{r,h}$. Moreover, because of Remark~\ref{rem:MaT}, we
    can assume as well that each $\ue$ is a $(\eta/2)$-strict subsolution on $Q^{\xb,\tb}_{r,h}$
    \footnote{This regularization step cannot be done if $\tb=\Tf$: this is why the comparison may only
    be proved on $\R^N\times [0,\Tf-\delta]$ for any $\delta>0$.}.

    \medskip

    \noindent\textbf{Step 4:} \emph{Properties of the regularized subsolution} -- Step 3 has two
    consequences 
    \begin{enumerate}
    \item[$(a)$] for any $\eps>0$ small enough, $\F^k(x,t,\ue,D \ue)\leq -\eta/2<0$ on $\Man{k}\cap
        \overline{Q^{\xb,\tb}_{r,h}}$ in a classical sense; 
    \item[$(b)$] since $\ue$ is an
        $(\eta/2)$-strict \wSSub of the $\psi$-Equation in
        $\mathcal{O}:=Q^{\xb,\tb}_{r,h}\setminus\Man{k}$ and since \LCR holds there because
        \PP{k+1} holds, we use the subdynamic programming principle for subsolutions (cf.
        Theorem~\ref{thm:sub.Qk.dpp}\index{Dynamic Programming Principle!for subsolutions}) which implies that
        each $\ue$ satisfies an $(\eta/2)$-strict dynamic programming principle in $\mathcal{O}$\footnote{We
        leave to the reader the careful checking that the proof of Theorem~\ref{thm:sub.Qk.dpp} uses only
        \PP{k+1} in $\mathcal{O}$ and never the $\F_*$-inequalities.}.
    \end{enumerate}
    These two properties allow us to have \LCR-evol in $Q^{\xb,\tb}_{r,h}$ in the final step.

    \medskip

    \noindent\textbf{Step 5:} \emph{Performing the local comparison} --
    From the previous step we know that for each $\eps>0$, $\ue$ satisfies the hypotheses of
    Lemma~\ref{lem:comp.fundamental} and we deduce from this lemma that 
    $$\forall (y,s)\in Q^{\xb,\tb}_{r,h} \;,\quad (\ue-v)(y,s)
    <\max_{\partial_p Q^{\xb,\tb}_{r,h}}(\ue -v)\;.$$
    Using that $u=\limssup \ue$, this yields a local comparison result (with inequality in the large
    sense) between $u$ and $v$ as $\eps\to0$.
	
    Therefore we have shown that \PP{k+1} implies \PP{k}, which ends the proof.
\end{proof}

\begin{remark} As it is clear in the above proof, the special structure of $\M$ does not play any
    role and time-dependent stratifications do not differ so much from time-independent ones. We
    remark anyway that a difference is hidden in the normal controllability assumption is that we
    cannot have a normal direction of the form $(0_{\R^N}, \pm 1)$ for $\Man{k}$ and this, for any $k$. 
\end{remark}

\section{Regular weak stratified subsolutions are strong stratified subsolutions}\label{rweqs}

As the title of the section indicates it, the main result is the
\begin{proposition}\label{wequals}
    If \HBASF holds, then any regular weak stratified subsolution is a strong stratified subsolution.
\end{proposition}

It is a little bit surprising to see that, provided \HBASF holds, the ``unnatural'' $\F_*,
(\F_{init})_*$ inequalities necessarily hold for regular weak stratified subsolution.
Besides, the rather indirect proof of Proposition~\ref{wequals} below---via the comparison
result---confirms how artificial these inequalities are. 

But, for the purpose of this book, this has a clear consequence: since \HBASF is a basic assumption which
is supposed to hold everywhere in this book, regular weak or strong stratified subsolutions make no
difference. For this reason, in the sequel, we almost only use the notion of \sSSub.

\begin{proof}
    We only show that a regular weak stratified subsolution satisfies the $\F_*$ inequality, the
    $(\F_{init})_*$ one being similar.

    Let $u$ be a regular weak stratified subsolution and $(x,t)\in \R^N\times (0,\Tf)$ be a strict
    local maximum point of $u-\phi$ where $\phi$ is a $C^1$-function in $\R^N\times [0,\Tf]$. We may
    assume \wlg that $u(x,t)=\psi(x,t)$. If $ \F_*(x,t,\phi(x,t),D\phi(x,t)\leq 0$, we are done.
    Hence we assume by contradiction that $ \F_*(x,t,\phi(x,t),D\phi(x,t)=2\delta>0$. 

    But $\F_*$ being \lsc and $\phi$ being smooth, $\F_*(y,s,\phi(y,s),D\phi(y,s)\geq \delta>0$ in
    $Q^{x,t}_{r,h}$ for $r,h>0$ small enough and we may also assume because of the strict local
    maximum point property that $u-\phi < 0$ on $\partial_p Q^{x,t}_{r,h}$. Since $\F_*\leq \F$,
    $\phi$ is a \SSup in $Q^{x,t}_{r,h}$ and the \LCR for the equation which holds as a by-product
    of the proof of Theorem~\ref{comp-strat-RN}---choosing perhaps smaller $r,h$---leads to 
    $$0=u(x,t)-\psi(x,t)\leq \max_{\partial_p Q^{x,t}_{r,h}} \,(u-\phi)<0\; ,$$
    \ie a contradiction to the maximum point property. Hence $ \F_*(x,t,\phi(x,t),D\phi(x,t)\leq 0$
    and the result is proved.  
\end{proof}

\chapter{Connections with Control Problems and Ishii Solutions}
\label{chap:stratcontr}    

\abstract{In this chapter, the connections between the notions of stratified solution to control problems and
classical Ishii solutions are studied. Given a set-valued map, the natural value function is
identified as the unique (strong) stratified solution of the associated Bellman Equation. As can be
expected, it corresponds to the minimal Ishii solution. Some partial results connecting Ishii and
stratified subsolutions are presented.}

\section{Value functions as stratified solutions}

\index{Control problem!in stratified domains}
In Section~\ref{Gen-DCP}, it is already shown that the value function $U$ defined by
$$U(x,t):=\inf_{\cT(x,t)}\Big\{\int_0^{+\infty}l\big(X(s),T(s)\big)\,e^{-D(s)}\ds\Big\}$$
is an Ishii supersolution of $\F=0$, therefore it is a stratified supersolution. But in order to get
the subsolution properties, the behavior of the dynamic is playing a key role via Assumptions
\TCBCL and \NCBCL.
\index{Stratified solutions!and control problems}

In the sequel, we treat in details the subsolution properties on $\Man{k}$ ($0\leq k\leq N+1$), \ie those
for $t>0$. The case $t=0$ and the corresponding inequalities on the $\Man{k}_0$ follow readily from the same arguments.

\begin{theorem}\label{SubP}\emph{--- Subsolution Properties.}\smsp
    Under assumption \HBASF, the value function $U$ is a regular weak stratified subsolution. More
    precisely,
    \begin{enumerate}
        \item[$(i)$] For any $k=0..N$, $U^* = (U|_{\Man{k}})^*$ on
        $\Man{k}$\,; 
        \item[$(ii)$] for any $k=0..(N+1)$,  $U^*$ is a Ishii subsolution of 
        $$\Fk(x,t,U^*,DU^*)=0\quad\text{on }\Man{k}\;.$$
    \end{enumerate}
\end{theorem}

In this result, we recall again that, for $k=0..N$, $(ii)$ is a viscosity inequality for an equation {\em
restricted to $\Man{k}$}, which means that if $\phi$ is a smooth function on
$\Man{k}$ (or equivalently on $\R^N\times(0,\Tf)$ by extension) and if $(x,t)\in
\Man{k}$ is a local maximum point of $U^*-\phi$ on $\Man{k}$, then
$$\Fk(x,t,U^*(x,t),D\phi(x,t))\leq 0\ \;.$$ 
This is why point $(i)$ is an important fact since it allows to restrict everything
(including the computation of the \usc envelope of $U$) to $\Man{k}$.

Of course, the case $k=N+1$ is particular since $\F^{N+1}=\F_*=\F$ on $\Man{N+1}$,
because \HBASF implies that $\F$ is continuous on $\Man{N+1}$.

This result already shows that $U^*$ is a weak stratified subsolution of the problem. But the
reader has probably already understood that, since \NCBCL holds, $U^*$ is also going to be a strong
stratified subsolution by Proposition~\ref{wequals}.

\begin{remark} 
    As we detail it in Section~\ref{sec:vfar}, the value function $U$ is \lsc and therefore regular
    on every $\Man{k}$ for $k=1..N$, \ie it satisfies \eqref{bonnelim11}. But unfortunately the
    lower semi-continuity does not provide the ``two-sided'' regularity \eqref{bonnelim12}. This is
    why Proposition~\ref{wequals} is required to have the right property on $\Man{N}$.
\end{remark}

\begin{proof} 
    Since all the results are local, we can assume \wlg that we are in the case of a flat $\Man{k}$,
    \ie if $(x,t)\in \Man{k}$ then, in a neighborhood of $(x,t)$, $\Man{k}=(x,t)+V_k$, a complete
    proof being obtained via a simple change of variable.

    \medskip

    \noindent\textbf{(a)} \emph{Proof of $(i)$ ---}
    For $k=0..N$, we consider $(x,t)\in \Man{k}$ and a sequence $(\xe,\te)\to (x,t)$ such that 
    $$ U^*(x,t) = \lim_{\e\to0}  U(\xe,\te)\; .$$ 
    We have to show that we can assume that $(\xe,\te) \in \Man{k}$. In all the
    sequel, we assume that $\e>0$ is small enough so that all the points remain in $B((x,t),r)$, the ball
    given by \NCBCL.

    Given a sequence $(x_\e,t_\e)$, we build a sequence $(\bar x_\e,\bar t_\e)_\e$ such that
    $(\bar x_\e,\bar t_\e) \in \Man{k}$ for any $\e$ and with $ U^*(x,t) = \lim_\e U(\bar x_\e,\bar
    t_\e)$. Notice that of course if $(x_\e,t_\e)$ already belongs to $\Man{k}$ we can set $(\bar
    x_\e,\bar t_\e)=(x_\e,t_\e)$ so let us assume that this is not the case.

    By Theorem~\ref{DPP}, for any solution $(X,T,D,L)$ of the
    differential inclusion starting from $(\xe,\te,0,0)$ and any $\theta>0$,
    $$U (\xe,\te) \leq \int_0^\theta
    l\big(X(s),T(s)\big)\exp(-D(s)) ds+U \big(X(\theta),T(\theta))\exp\big(-D(\theta)\big)\;.$$
    Now let $(\tilde x_\e,\tilde t_\e)$ be the projection of $(\xe,\te)$ onto $\Man{k}$ and let us
    denote by $n_\e$ the vector $n_\e:= (\tilde x_\e,\tilde t_\e)-(\xe,\te) \in V_k^\bot$.
    Using \NCBCL we know that, for any $(y,s) \in B((x,t),r)$, there exists $b\in\B(y,s)$ with normal
    component $2^{-1}\delta.n_\e  |n_\e |^{-1}\in B(0,\delta)$. More precisely, there exist $b$ which
    can be decomposed as $b=b_\top + b_\bot $ with $b_\top\in V_k$, $b_\bot\in V_k^\bot$,
    and $b_\bot:= 2^{-1}\delta.n_\e  |n_\e |^{-1}$. We denote by $\widetilde{\BCL}(y,s)$ the set of all
    $(b,c,l) \in \BCL(y,s)$ for which $b$ is of this form.
    
    Clearly, the map $(x,t)\mapsto\widetilde{\BCL}(y,s)$ has compact, convex images and is upper
    semi-continuous. Solving the associated differential inclusion starting from
    $(\xe,\te)$, we get a solution $(X,T,D,L)$ such that $(X(s),T(s))\in B((x,t),r)$ for $s$ small
    enough, independent of $\e$. Moreover, for $s_\e=2|n_\e |/\delta$, 
    $$(\bar x_\e,\bar t_\e) = (X(s_\e),T(s_\e))=(\xe,\te)+s_\e b=(\tilde x_\e +
    y_\e,\tilde t_\e + \tau_\e)\; ,$$ 
    where $(y_\e,\tau_\e) \in V_k$, $|(y_\e,\tau_\e)| = O(|\tilde x_\e-\xe|+|\tilde t_\e-\te|)$. 
    Indeed, $s_\e b_\bot = n_\e$ and therefore $(\xe,\te)+s_\e b \in (\tilde x_\e,\tilde t_\e) +V_k$.

    Therefore, $(\bar x_\e,\bar t_\e) \in \Man{k}$ since $\Man{k}=(x,t)+V_k$ in a neighborhood of
    $(x,t)$ and using the Dynamic Programming Principle above with $\theta = s_\e$ yields
    $$U (\xe,\te) \leq O(s_\e)+U \big(X(s_\e),T(s_\e)\big)\exp(-D(s_\e))= 
    O(s_\e)+U \big(\bar x_\e,\bar t_\e \big)(1+ O(s_\e))\;.$$ 
    Finally since $s_\e \to 0$ as $\e \to 0$, we deduce that 
    $$ \limsup_{\e\to0} U \big(\bar x_\e,\bar t_\e\big) \geq \limsup_{\e\to0} \, 
    U (\xe,\te)=U^*(x,t)\; ,$$ 
    which shows $(i)$ since $(\bar x_\e,\bar t_\e\big) \in \Man{k}$.

\medskip

    \noindent \textbf{(b)} \emph{Proof of $(ii)$ --} As we already mentioned above, 
    the result for $k=N+1$ is given by Theorem~\ref{thm:SubP}.
    Hence it remains to examine the cases $k=0..N$.
    
    For such $k$, let $\phi$ be a smooth function on $\Man{k}$ and
    let $(x,t)\in \Man{k}$ be a local maximum point of $U^*-\phi$ on $\Man{k}$, we have to show that
    $$\Fk(x,t,U^*(x,t),D\phi(x,t))\leq 0\ \;.$$ 

    Using $(i)$, we can consider a sequence $(\xe,\te) \in \Man{k}$ such that $U(\xe,\te)\to U^*(x,t)$ and use 
    Theorem~\ref{DPP}, which implies
    \begin{equation}\label{dpp-xete}
    U (\xe,\te) \leq \int_0^\theta
    l\big(X(s),T(s)\big)\exp(-D(s)) ds+U \big(X(\theta),T(\theta)\big)\exp\big(-D(\theta)\big)\ ,
    \end{equation}
    for any solution $(X,T,D,L)$ of the differential inclusion starting from $(\xe,\te,0,0)$.

    But now we can use the result of Lemma~\ref{tgfields}: for any $(b,c,l)\in \BCL^k (x,t)$ and
    $\eta >0$, $\BCL^k (y,s)\cap  B((b,c,l),\eta) \neq \emptyset$ if $(y,s)$ is close enough to
    $(x,t)$. Solving locally the differential inclusion with $\BCL^k (y,s)\cap  B((b,c,l),\eta) $
    instead of $\BCL$ and using the associated solution in \eqref{dpp-xete} allows to obtain the
    viscosity inequality for $(b,c,l)$ as in the standard case.

    Since this is true for any $(b,c,l)\in \BCL^k (x,t)$, the result is complete.
\end{proof}

An immediate consequence of Theorem~\ref{SubP}---and of its analogue for $t=0$---, using also Theorem~\ref{comp-strat-RN} and Proposition~\ref{wequals}, is the
\begin{corollary} 
    Under the assumptions of Theorem~\ref{SubP}, the value function $U$ is continuous in $\R^{N}\times [0,\Tf[$ 
    and is the unique (strong) stratified solution of the Bellman Equation.  
\end{corollary}

\section{Stratified solutions and classical Ishii viscosity solutions}
\label{sec:Strat-Is}

The aim of this section is to compare the two notions of solutions, in particular under the
assumptions of Theorem~\ref{comp-strat-RN}. Of course, (weak or strong) stratified solutions and classical
Ishii viscosity solutions can coincide only when the latter are uniquely identified and the case of
codimension-$1$ discontinuities shows that this clearly requires some additional assumptions, \cf
Part~\ref{part:codim1}.

We present two kind of results in this section: the first one, which is just an easy remark, is that
the stratified solution is the minimal Ishii (super)solution; the second one provides a particular
case where we can show that Ishii subsolutions are (strong) stratified subsolutions.

\subsection{The stratified solution as the minimal Ishii solution}

Before addressing the question of identifying conditions under which classical
Ishii viscosity solution and stratified solution coincide, we begin with an
easy consequence of Theorem~\ref{comp-strat-RN}.
\begin{corollary}\label{Strat-eq-Ishii} If \HBASF holds, the unique stratified solution of
    Equation~\eqref{eq:super.H.strat} is also the minimal Ishii viscosity supersolution and solution
    of Equation~\eqref{eq:super.H.strat}.
\end{corollary}
The proof of this result is obvious since Ishii viscosity supersolutions and
stratified supersolutions are the same; therefore
Corollary~\ref{Strat-eq-Ishii} is a straightforward application of
Theorem~\ref{comp-strat-RN}.

In the case of codimension-$1$ discontinuities (see Part~\ref{part:codim1}),
Corollary~\ref{Strat-eq-Ishii} implies that $\Um$ is the unique stratified
solution and actually the reader can check that
Theorem~\ref{thm:minimal.charac} is nothing but the first uniqueness (and
comparison) result for a stratified solution in this book, $\HT$ providing the
subsolution inequality on $\Man{N}$.

\subsection{Ishii subsolutions as stratified subsolutions}

The next very natural question is: under which conditions can it be proved that
a classical Ishii viscosity subsolution is a stratified (strong) subsolution? Of course,
this question is meaningful only for subsolutions since the supersolutions are
the same. Notice that when this is the case, we conclude that uniqueness holds for the Ishii
formulation since the unique stratified subsolution is also the unique classical Ishii viscosity
solution. 

This question then appears as a generalization in the direction of looking for
conditions which ensure, in one dimension, that $\Up\equiv\Um$. A partial but rather general
answer is given by Lemma~\ref{lem:H1m.H2p.c}. The reader can check on examples
that this lemma is of a rather simple use as it can be seen on
Chapter~\ref{chap:KPP}.

In the more complicated framework of stratified problems, we are also looking
for simple conditions which can easily be checked for more general types of
discontinuities. The ones we propose in this section are unavoidably rather
restrictive but they cover anyway some interesting cases as we will illustrate
below by several examples.

As in the previous section, we treat only the case of the subsolution inequalities on the $\Man{k}$,
\ie those for $t>0$, but similar arguments gives the same results for the $\Man{k}_0$ if $t=0$.

\medskip

A way to obtain the $\F^k$-subsolution inequalities by using the Ishii subsolution
condition $\F_*\leq0$ consists in\\[2mm]
$(i)$ ``shifting'' $\Man{k}$ into some $\Man{k}_\e\subset\Man{N+1}$ such that
$\Man{k}_\e\to\Man{k}$;\\[2mm]
$(ii)$ using the Ishii inequality $\F_*\leq 0$ on $\Man{k}_\e$, since it is contained in $\Man{N+1}$, in order 
to obtain a $\F^k_\e \leq 0$-inequality on $\Man{k}_\e$;\\[2mm]
$(iii)$ passing to the limit through a stability property to get $\F^k\leq0$ on $\Man{k}$ as $\e\to0$.

In order to perform this stability strategy in a quite general way we need to take into
account the fact that $\Man{k}$ may be approached by several $\Man{k,i}_\e\subset\Man{N+1}$, and of
course we only need to choose one which yields the result. We refer to the example
below to better understand this remark. Notice that using stability in the stratified setting is by no
means a routine exercise as it is in the classical continuous case, \cf Chapter~\ref{NESR}, but here we use
a ``simple'' stability result since the $\F^k_\e$-inequalities are set on the $\Man{k}_\e$, which are just
copies of $\Man{k}$ and which converge to $\Man{k}$ in a strong enough way. 

\

To be more precise, let us assume that $\mathcal{O}\subset\R^{N+1}$ is an open set and that
for $(x,t)\in\mathcal{O}$,
$$\F(x,t,r,p)=\sup_{(b,c,l)\in\BCL(x,t)}\big\{ -b\cdot p + c r- l \big\}$$ 
for some set-valued map $\BCL:\mathcal{O} \to\mathcal{P}(\R^{N+3})$.

\begin{definition}
    For $k=0..N$, we say that $\Man{k}$ is locally $\Man{N+1}$-approached by a family of $k$-dimensional manifolds
    $(\Man{k,i}_\e)_{i\in\mathcal{I}}$ at $(x,t)\in \Man{k}$ if there exists $r>0$ such that
    $B((x,t),\bar r) \subset \mathcal{O}$ and for all $i\in\mathcal{I}$,
    $$\begin{cases}
   \Man{k,i}_\e \cap B((x,t),\bar r) \subset \Man{N+1}\;,\\ 
   \Man{k,i}_\e \cap B((x,t),\bar r) \to \Man{k} \cap B((x,t),\bar r)\text{ in the $C^1$-topology\;.}
    \end{cases}$$
In such a situation, for $(y,s) \in \Man{k,i}_\e$, we set
\begin{equation*}\label{eq:def.ham.eps}
    \F^{k,i}_\e (y,s,r,p): = \sup_{\substack{(b,c,l)\in\BCL(y,s)\\ b\in T_{(y,s)}
   \Man{k,i}_\e}}\big\{ -b\cdot p +cr - l\big\}\;.
\end{equation*}
\end{definition}

Our result using this notion is the

\begin{proposition}\label{IequalS} 
    Let us assume that \HBASF holds and that $u$ be an \usc classical
    Ishii viscosity subsolution of
    $$ \F(x,t,u,Du)\leq0\quad \hbox{in  }\mathcal{O}\subset\R^{N+1}\;.$$
    We also assume that $\Man{k}$ is $\Man{N+1}$-approached by $\displaystyle (\Man{k,i}_\e)_{i\in\mathcal{I}}$ at $(x,t)$.
    If
    \begin{equation}\label{eqn:fk-appr}
\F^k \leq \max_{i\in\mathcal{I}}\big(\liminf \F^{k,i}_\e\big)\text{  in }\Man{k} \cap
    B((x,t),\bar r)\;,
\end{equation}
    then 
    $$\F^k(x,t,u,Du)\leq 0\text{ in }\Man{k} \cap B((x,t),\bar r)\;.$$
\end{proposition}

The idea of this proposition is very simple and follows the above described program: if, on the
``shifted'' $\Man{k,i}_\e$, the (approximate) $\F^{k,i}_\e$-inequalities follow from the $\F_*$-one,
we can conclude by ``stability'' that $\liminf \F^{k,i}_\e\leq 0$ on $\Man{k}$. Then by using any possible choice of
these ``shifted'' families $(\Man{k,i}_\e)_{i,\e}$, it is enough to have \eqref{eqn:fk-appr} in order
to conclude that the $\F^k$-inequality holds.

We point out that this result takes a simpler form in the \AFS case since we can typically use
manifolds $\Man{k,i}_\e$ which are nothing but $e_i^\e + \Man{k}$ for some suitable choice of
$e_i^\e \in \R^{N+1}$, \ie on some copy of $\Man{k}$ which is {\em included in $\Man{N+1}$}.
Typically, we need at least one of these copies on each connected components of
$\Man{N+1}\setminus\Man{k}$. Of course, we are in a similar setting in the \TFS case after a
suitable change of variables.

From a control point-of-view, the interpretation of Proposition~\ref{IequalS} is the following:
the best strategy to stay on $\Man{k}$ is to use tangential dynamics which already exist in one of the
connected components of $\Man{N+1}$, without combining incoming or outgoing dynamics coming from
several of these connected components. Such situation clearly leads to $\VFp=\VFm$ in the two-domains case.

We also insist on the fact that we have treated the case of $\F$ but an analogous result also holds for
$\F_{init}$.

\begin{proof} 
    Using Proposition~\ref{ajout} later in this book, the
    $\F^{k,i}_\e$-inequalities on the $ \Man{k,i}_\e$ are direct consequences of the $\F$-one on
    $\Man{N+1}$. Hence $\F^{k,i}_\e\leq 0$ on $\Man{k,i}_\e \cap B((x,t),\bar r)$ and, by stability
    we deduce that $\liminf \F^{k,i}_\e\leq 0$ on $\Man{k} \cap B((x,t),\bar r)$.  This implies that
    $\max_i (\liminf \F^{k,i}_\e) \leq 0$ on $\Man{k} \cap B((x,t),\bar r)$ and the result follows.
\end{proof}

An interesting direct consequence is the 
\begin{corollary}\emph{--- Equivalence of Ishii and stratified solutions.}\smsp
    If the assumptions of Proposition~\ref{IequalS} hold for any $(x,t) \in \Man{k}$ for $k=1..N$,
    then any Ishii subsolution in $\R^N \times (0,\Tf)$ is a stratified subsolution. As a
    consequence, Ishii sub and supersolutions are the same as stratified sub and supersolutions in
    $\R^N \times (0,\Tf)$.  
\end{corollary}

We have decided to restrict Proposition~\ref{IequalS} and this corollary to the domain $\R^N \times
(0,\Tf)$ but, of course, similar results can be obtained at $t=0$ and for general stationary
problems. We leave these generalization to the reader.

\

The example below shows a simple situation where the result can be applied, leading to uniqueness
for the Ishii problem.

\begin{example}\label{ex:ishii.strat}
    {\rm  We consider the equation
    $$u_t + a(x)|Du|=g(x)\quad \hbox{in  }\R^2\times (0,\Tf)\; ,$$
    where $a=a_i$ and $g=g_i$ in $\Omega_i$ where the $\Omega_i$ are in Figure~\ref{fig:cross} and
    the functions $a_i, g_i$ are continuous. Of course, we assume that $a_i(x) \geq 0$ for
    any $x\in \Omega_i$, $1\leq i \leq 4$.

    \medskip

    \noindent\textbf{(a)}
    Let us first consider $\Man{2}$ and the part $\{x_1=0, x_2>0\}$. Here 
    $$\F^2 (x,t,(p_x,p_t))= p_t + \sup\, \left\{ -(\theta a_1 v_1 +
    (1-\theta)a_2v_2)\cdot p_x -(\theta g_1 +(1-\theta)g_2)\right\}\; ,$$
    where the supremum is taken on all $|v_1|, |v_2|\leq 1$ and all $0\leq \theta \leq 1$ such
    that\\ $(\theta a_1 v_1 +(1-\theta)a_2v_2)\cdot e_1 = 0$.

    It is obvious that $\F^2$ can be computed by choosing $v_1,v_2$ such that
    $v_1\cdot e_1=v_2\cdot e_1 =0$, \ie by taking dynamics which are in the direction of $\Man{2}$
    and writing 
    $$-(\theta a_1 v_1 +(1-\theta)a_2v_2).p_x -(\theta g_1 +(1-\theta)g_2)=\theta (- a_1
    v_1\cdot p_x-g_1) +(1-\theta)(-a_2v_2\cdot p_x -g_2)\; ,$$
    it follows that 
    $$\F^2 (x,t,(p_x,p_t))=\max(p_t +a_1|(p_x)_2|-g_1, p_t +a_2|(p_x)_2|-g_2)\; ,$$
    where $(p_x)_2$ is the second component of $p_x$, \ie the tangential part of the gradient in
    space.

    \medskip

    \noindent\textbf{(b)}
    Examining the condition to be checked for Proposition~\ref{IequalS}, we see that we can choose $\Man{2,i}=(-1)^{i+1}\e e_1+\Man{2}$
    and for $\F^{2,i}_\e$, we have
$$
        \F^{2,i}_\e (x+(-1)^{i+1}\e e_1 ,t,(p_x,p_t)) =p_t +a_i(x+(-1)^{i+1}\e e_1)|(p_x)_2|-g_i(x+(-1)^{i+1}\e e_1)\; ,$$
 and $\G_i(x ,t,(p_x,p_t)) =p_t +a_i(x)|(p_x)_2|-g_i(x)$. Therefore we have $\F^2=\max(G_1,G_2)$, we can apply the result and of course the same property holds for the three other
    parts of $\Man{2}$.

    \medskip

    \noindent\textbf{(c)}
    For $\Man{1}:=\{(0,0)\}\times (0,\Tf),$ the checking is even simpler by considering $(0,0)\pm \e e_1\pm \e e_2$, one can easily check that the $\F^1$
    condition, \ie $p_t-\min(g_i)\leq 0$, is satisfied.}
\end{example}

\begin{remark}
    As the above example shows, the result of Proposition~\ref{IequalS} is not very sophisticated
    but it has the advantage to be very simple to apply.
\end{remark}

\section{Concrete situations that fit into the stratified framework}

In this section, we consider stratified problems through a different point of view, maybe closer to
concrete applications. We give general frameworks in which \HBASF, \NCe/\NCBCL and \TC/\TCBCL
are satisfied so that the connections between the control and pde approaches are satisfied.  

In the following two subsections, we assume that we are given a tangentially flattenable stratification
$(\Man{k})_k$ of $\R^N$ and each manifold $\Man{k}$ is written as the union of its connected
components $\Man{k,j}$ $$ \Man{k} = \bigcup_{j=1}^{J(k)}\, \Man{k,j}\; ,$$ 
where $J(k)\in \N \cup\{+\infty\}$.

\subsection{A general control-oriented framework}

Here we start from a collection of specific control problems on each $\Man{k,j}$.

\medskip

\noindent\textbf{The control problems ---}
On each $\Man{k,j}$, we are given a space of control $A_{k,j}$ and functions $(b^{k,j}, c^{k,j},
l^{k,j})$ representing the dynamic, discount factor and cost for a control problem on $\Man{k,j}$.
For the sake of simplicity, we assume that all these function are defined in $\R^N\times [0,
\Tf]\times A_{k,j}$ with the condition $b^{k,j} (x,t,\alpha_{k,j}) \in T_x \Man{k}$ for any
$(x,t)\in \Man{k,j}$ and $\alpha_k \in A_k$ in order that the dynamic preserves $\Man{k,j}$ at least
for a short time. 

\medskip

\noindent\textbf{The Hamiltonians --- }
If $(x,t)\in \Man{k,j}$, we introduce the associated Hamiltonian 
$$ \tilde H^{k,j}(x,t,r,p):= \sup_{\alpha_{k,j} \in A_{k,j}}\,\left\{-b^{k,j}(x,t, \alpha_{k,j})\cdot p +
c^{k,j}(x,t, \alpha_{k,j})r-l^{k,j}(x,t, \alpha_{k,j})\right\}\; ,$$
which is defined for $r\in \R$ and a priori only for $p\in T_x \Man{k}$ but we can as usual extend
this definition for $p\in \R^N\times \R$.

If $(x,t) \in \R^N \times (0,\Tf)$, setting $L(x,t):=\{(k,j) ;\ (x,t) \in \overline{\Man{k,j}}\}$
and define
$$\BCL(x,t)=\overline{\rm Conv}\left\{\bigcup_{{(k,j)} \in L(x,t)} 
\{(b^{k,j}, c^{k,j}, l^{k,j})(x,t, \alpha_{k,j}),\ \alpha_{k,j} \in A_{k,j}\} \right\}\;,$$
$$
\F(x,t,r,p)= \sup_{\substack{(k,j) \in L(x,t),\\ \alpha_{k,j} \in A_{k,j}}}
\Big\{  -b^{k,j}(x,t, \alpha_{k,j})\cdot p + c^{k,j}(x,t, \alpha_{k,j})r - 
l^{k,j}(x,t, \alpha_{k,j})\Big\}\; .$$

\noindent\textbf{Assumptions ---}
In order to have Assumption \TC satisfied, it is enough that each $(b^{k,j}, c^{k,j}, l^{k,j})$
satisfies \hyp{BACP} and for \NCe, we have to assume that if $(x,t)\in \Man{\bar k}$, then the set
$$\overline{\rm Conv}\left(\bigcup_{\substack{(k,j) \in L(x,t),\\ k>\bar k}} \{(b^{k,j}, c^{k,j},
l^{k,j})(x,t, \alpha_{k,j}),\ \alpha_{k,j} \in A_{k,j}\} \right)\; ,$$
 satisfies \NCBCL (instead of $\B$).

\subsection{A general pde-oriented framework}

On the contrary, here, we start from a general equation and define all the $\F^k$ by induction.
Unfortunately this pde-oriented example will not be completely formulated in terms of pde and
Hamiltonians, the difficulty being analogous to defining $\HT$ in Part~\ref{part:codim1}. To
simplify, we treat the case when the stratification does not depend on times, \ie
$\Man{k+1}=\tMan{k}\times (0,\Tf)$ for all $0\leq k \leq N$, where $ (\tMan{k})_{k}$ is a
stratification--a \TFS--of $\R^N$.

\medskip

\noindent\textbf{The case $k=N+1$ ---}
We start from $\Man{N+1}$ which we write as the union of its connected components
$$ \Man{N+1} = \bigcup_{j=1}^{J(N+1)}\, \tMan{N,j} \times (0,\Tf)\; .$$
We consider the case when
$$ \F^{N+1}(x,t,r,(p_x,p_t))=p_t + \tilde H^{N,j}(x,t,r,p_x)\quad\hbox{in  }\tMan{N,j}\times (0,\Tf)\; ,$$
for all $j$ where the Hamiltonians $\tilde H^{N,j}$ are defined by
$$ \tilde H^{N,j}(x,t,r,p)= \sup_{\alpha_{N,j} \in A_{N,j}}\left\{  -b^{N,j}(x,t, \alpha_{N,j})\cdot p + c^{N,j}(x,t, 
\alpha_{N,j})r - l^{N,j}(x,t, \alpha_{N,j})\right\}\; ,$$
where the control sets $A_{N,j}$ are compact metric spaces. A simple but natural situation is when all these 
Hamiltonians can be extended as continuous in $\R^N\times [0,\Tf]$ functions satisfying \hyp{BA-HJ}. These 
Hamiltonians are the analogues of $H_1,H_2$ in Part~\ref{part:codim1}.

\medskip

\noindent\textbf{Induction for $k<N+1$ ---}
It remains to define $\F$ and $\F^{k+1}$ on all $\tMan{k} \times (0,\Tf)$ for $k<N$ and this has to
be done by induction.  For $k=N-1$, if
$$ \Man{N} = \bigcup_{j=1}^{J(N)}\, \tMan{N-1,j}\times (0,\Tf)\; ,$$
we can assume that, on each $\tMan{N-1,j}\times (0,\Tf)$, we have an Hamiltonian $\tilde H^{N-1,j}$
and we have, for any $(x,t) \in \tMan{N-1,j} \times (0,\Tf)$
$$ \F(x,t,r,(p_x,p_t))=\max_{l\in L(x,t)} \Big( p_t + \tilde H^{N,l}(x,t,r,p_x), p_t + 
\tilde H^{N-1,j}(x,t,r,p_x)\Big)\; ,$$
with $L(x,t):=\{l ;\ (x,t) \in \overline{\tMan{N,l}}\times (0,\Tf)\}$. On the other hand, $\F^N$ may
be decomposed into two parts: the analogue of the $\HT$-one in Part~\ref{part:codim1} coming from
$\F^{N+1}$ and the specific $\tilde H^{N-1,j}$-one reflecting a particular control problem on
$\tMan{N-1,j}\times (0,\Tf)$. This means
$$ \F^N (x,t,r,(p_x,p_t))=\max \Big( \F^{N+1}_T (x,t,r,(p_x,p_t)), 
p_t + \tilde H^{N-1,j}(x,t,r,p_x)\Big)\; ,$$
where $\F^{N+1}_T (x,t,r,(p_x,p_t))$ is built in the following way: as in the previous section, we set
$$ \overline{\rm Conv}\Big(\bigcup_{l \in L(x,t)} 
\{(b^{N,j}, c^{N,j}, l^{N,j})(x,t, \alpha_{N,j}),\ \alpha_{N,j} \in A_{N,j}\} 
\Big)\; ,$$ 
and, for $(x,t)\in \tMan{N-1,j}\times (0,\Tf)$ we denote by $\BCL^{N-1}_T(x,t)$ the subset of $(b,c,l)$ in this 
closed convex envelope such that $b\in T_{x} \tMan{N-1,j}$. Then
$$ \F^{N+1}_T (x,t,r,(p_x,p_t))=p_t + \sup_{\BCL^{N-1}_T(x,t)}\left\{  -b\cdot p_x + cr - l\right\}\; .$$

For any $k$, the construction is analogous. For any connected component of $\tMan{k,j} \times (0,\Tf)$ of $\tMan{k}
\times (0,\Tf)$, $\F$ and $\F^{k+1}$ are constructed in the same way by using, for $\F$, a maximum of the $\F^{k+2}, 
\F^{k+3}, \cdots, \F^{N+1}$ nearby and of $p_t +  \tilde H^{k,j}(x,t,r,p_x)$ where $\tilde H^{k,j}$ is a specific 
Hamiltonian on $\Man{k,j} \times (0,\Tf)$, while for $\F^{k+1}$, one has to built a tangential Hamiltonian $\F^{k+2}_T$ 
and take the maximum with $p_t +  \tilde H^{k,j}(x,t,r,p_x)$. The construction of $\F^{k+2}_T$ is the same as in the
previous section and is based on computing the element of $\BCL(x,t)$ for $(x,t) \in \tMan{k,j} \times (0,\Tf)$ coming 
from $\tMan{\bar k,l} \times (0,\Tf)$ for $\bar k >k$ and for the nearby connected components of the $\tMan{\bar k} 
\times (0,\Tf)$.

\chapter{Stability Results}
\label{NESR}

\abstract{This chapter provides several stability results for stratified solutions: in the most
complete one, new parts of the stratification can appear at the limit, while other of them can
disappear.}


Stability results are, of course, a fundamental feature of viscosity solutions.  A general stability
result for solutions in the Ishii sense (\cf Theorem~\ref{hrl}) is readily available for
Hamiltonians with any type of discontinuities. Therefore it can be used for supersolutions in the
stratified case, stratified supersolutions being nothing but ordinary Ishii supersolutions.

But clearly the case of subsolutions is far more complicated: passing to the limit in all the
viscosity inequalities $\F^k\leq 0$ on $\Man{k}$ for all $k=0,..,N+1$ creates difficulties both at
the level of the Hamiltonians and the stratification. 

For the Hamiltonians, assuming there is sequence of Hamiltonians $(\F_\e)_\e$ all associated to the
same stratification and $\limssup \F_\e=\F$, first it is not clear that $\limiinf \F_\e=\F_*$; but
it is even less clear that for $k=0..N$, $\limiinf \F_\e^k = \F^k$, or at
least  $\limiinf \F_\e^k \geq \F^k$, which would be sufficient to get standard subsolution
inequalities for the limiting problem. We point out here that the notion of \wSSub may drop the
first difficulty---which is already an important result---but the second one is, of course,
unavoidable.

For the stratification, there are two levels of difficulties: either we just want to take into
account cases for which the local structure of the stratification is unchanged, \ie the discontinuities
are the same, they are just slightly moved; or we wish to treat cases where some parts of the
stratification are created or deleted, \ie some discontinuities may appear or disappear in the
Hamiltonians.

There is also a last difficulty, connected to the half-relaxed limits method: in order to use it in
the easiest way, one wants to use the $\limssup$ related to $\R^N\times (0,\Tf)$. But then,
passaging to the limit in the $\F^k$-inequalities becomes a problem: even if we consider problems with a
fixed stratification, we cannot simply use the standard stability result on $\Man{k}$ since it
relies on the $\limssup$ related to $\Man{k}$, not to $\R^N\times (0,\Tf)$. This difficulty which
looks like the one we encounter in control problems (\cf Theorem~\ref{SubP}-$(i)$) is solved by the
usual \emph{normal controllability assumption}.

In this chapter, we address all these difficulties: we first provide a basic stability result in the
case where the structure of the stratification is unchanged, without solving the difficulty related
to the convergence of the Hamiltonians. Then we show how the difficulty related to the convergence
of the Hamiltonians can be treated and finally we show how to take into account some modifications in
the structure of the stratification, both when new parts appear and when some parts disappear in the limit.

We conclude this introduction on stability results by a warning: as it will be clear in some of the
applications, there are often simpler method to show that a limit of stratified (sub)solutions is a
stratified (sub)solution. The aim of this chapter is more to give ideas on the various possibilities
that may be considered than to give results that can be applied blindly.

\section{Strong convergence of stratifications when the local structure is unchanged}

As it is clear from the above introduction, a stability result for a stratified problem requires two
ingredients; first a suitable notion of convergence for stratifications and then some
assumptions on the convergence of the Hamiltonians. But, of course, these ingredients should be
compatible enough to lead to a stability result.

Let us start from a definition given in \cite{AEYW} for the convergence of locally flattenable stratifications
which we adapt to the case of stratifications of $\R^N \times (0,\Tf)$. \index{Stratification! convergence of}
\label{page:strat.conv.un}
\begin{definition}\label{def:conv.strat-art}\emph{--- Strong convergence of locally flattenable
    stratifications.}\smsp
    We say that a sequence $(\M_\eps)_\eps$ of locally flattenable stratifications of $\R^N\times (0,\Tf)$
    converges to a locally flattenable stratification $\M$ if, for each $(x,t)\in\R^N\times (0,\Tf)$, there exists
    $r>0$, an \AFS $\M^\star=\M^\star((x,t),r)$ in $\R^N\times (0,\Tf)$, a change of coordinates $\Psi^{x,t}$ as in
    Definition~\ref{def:LFS} and, for any $\eps >0$, a family of changes of
    coordinates $\Psi^{x,t}_\eps$ as in Definition~\ref{def:LFS}\ 
    \footnote{without imposing $\Psi^{x,t}_\eps(x,t)=(x,t)$ and $\Psi^{x,t}(x,t)=(x,t)$} satisfying
    \begin{enumerate}
        \item[$(i)$] for any $0\leq k \leq N$, if $\Man{k}\cap B((x,t),r) \neq \emptyset$, then
        $\Psi^{x,t}(\Man{k}\cap B((x,t),r))=\M^\star\cap\Psi^{x,t}(B((x,t),r))\,$ and, for any $\e>0$, 
            $\Psi^{x,t}_\eps(\Man{k}_\eps\cap B((x,t),r))= \M^\star \cap \Psi^{x,t}_\eps(B((x,t),r))$.
             
        \item[$(ii)$] the changes of coordinates $\Psi^{x,t}_\eps$ converge in $C^{1}(B((x,t),r))$ to
            $\Psi^{x,t}$ and their inverses $(\Psi^{x,t}_\eps)^{-1}$ defined on $\Psi^{x,t}(B((x,t),r))$ also
            converge in $C^1$ to $(\Psi^{x,t})^{-1}$.
    \end{enumerate}
    We denote this convergence by $\M_\eps \toLFSs \M$ where {\rm LFS} stands for Locally Flattenable
    Stratification and ``$s$'' for ``strong'' convergence.   
\end{definition}

This definition essentially means that a sequence $(\M_\eps)_\eps$  of stratification converges to
$\M$ if the $\M_\eps$ are locally just {\em smooth}, little deformations of $\M$. Indeed,
$$\Man{k}_\eps\cap B((x,t),r)= [\Psi^{x,t}_\eps]^{-1}\circ \Psi^{x,t}\Big(\Man{k}\cap
B((x,t),r)\Big)\quad\text{and}\quad [\Psi^{x,t}_\eps]^{-1}\circ\Psi^{x,t}\to \mathrm{Id}\text{ in }C^1\;.$$
Technically, this allows to work locally with a fixed stratification $\M^\star$, removing completely
the difficulty of the convergence of stratification which is easily described by the convergence
of $\Psi^{x,t}_\eps$ to $\Psi^{x,t}$.

Of course, in this definition, we can replace ``locally flattenable stratification'' by 
``tangentially flattenable stratification'' and \AFS by \TFS without changing the global idea of this convergence and in this
more general case, we will just write $\M_\eps \toTFSs \M$.

\medskip

\noindent\textbf{Drawbacks ---}
Unfortunately this definition---even with the above generalization---excludes a lot of interesting
cases, the first one being the regularization of a corner, see Figure~\ref{fig:book}: in $\R^3$, we
define $\M$ by
$$ \Man{1}:=\{(0,0,x_3),\, x_3 \in \R\}\; ,\; 
\Man{2}:=\{(x_1,|x_1|,x_3),\, x_1\in \R \setminus\{0\},\, x_3 \in \R\}\, ,$$
and $\Man{0}=\emptyset$, $\Man{3}=\R^3 \setminus\left(\Man{1}\cup \Man{2}\right)$. 
Defining $\M_\eps$ through
$$\Man{2}_\e :=\{(x_1,(x_1^2+\e^2)^{1/2}-\e ,x_3),\, x_1\in \R \setminus\{0\},\, x_3 \in \R\}\, ,$$
and with $\Man{0}_\e= \Man{0}$, $\Man{1}_\e= \Man{1}$ and $\Man{3}_\e= \R^3 
\setminus\left(\Man{1}\cup \Man{2}_\e\right)$,
we see that we cannot expect the convergence of $\M_\e$ in the sense of the above definition.
Indeed, the dashed axis on Figure~\ref{fig:book} which should converge to the $x_3$-axis of the
limiting stratification does not exist in the approximating stratifications.

\begin{figure}[h!]
    \begin{center}
   \includegraphics[width=0.4\textwidth]{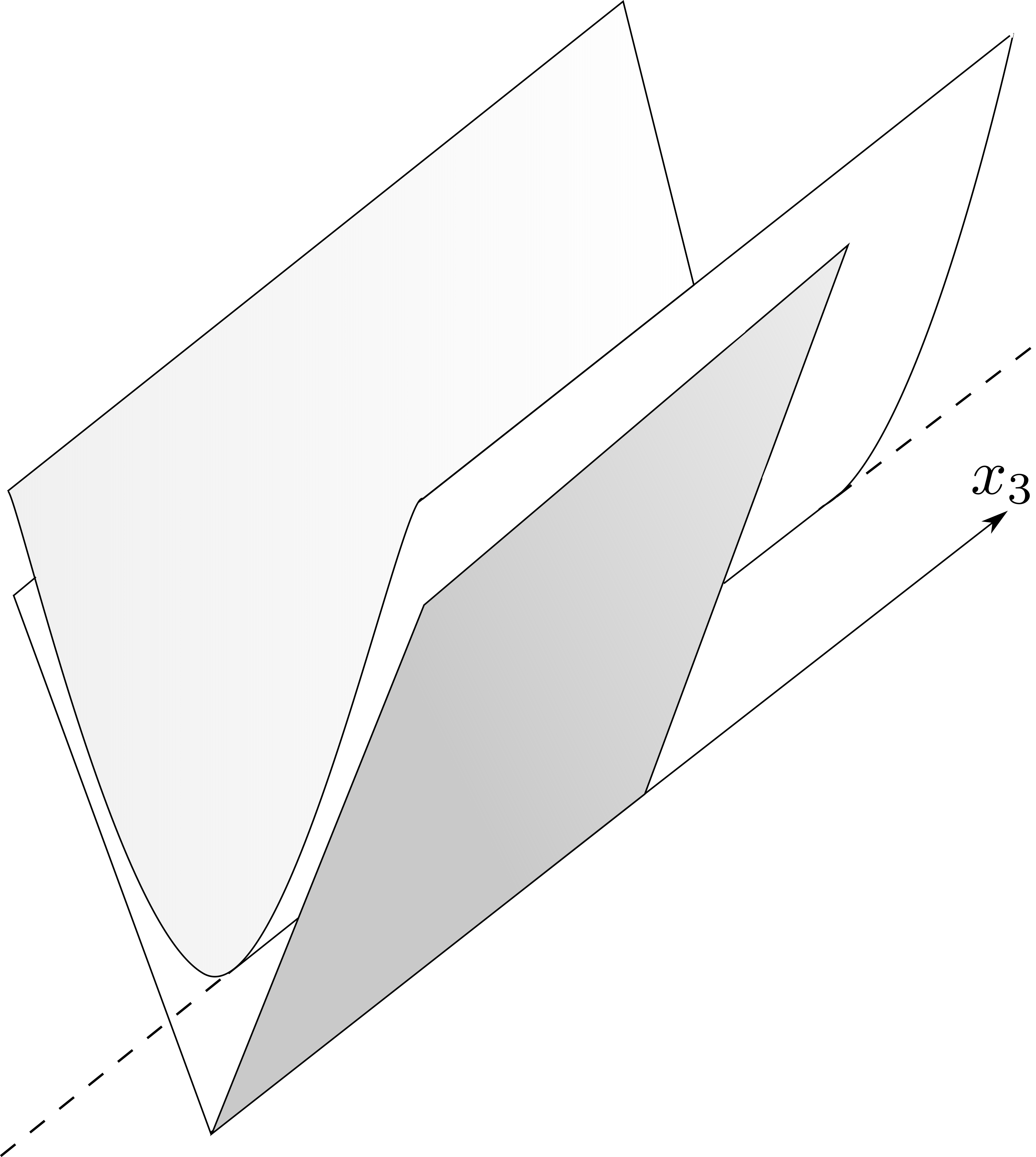}
   \caption{The "book" approximation}
   \label{fig:book} 
   \end{center}
\end{figure}

Another approach uses the following approximation of $\M$: we set $\Man{0}_\e= \emptyset$,
$$ \Man{1}_\e :=\{(\e,0,x_3),\, x_3 \in \R\}\cup \{(-\e,0,x_3),\, x_3 \in \R\}\; ,$$
$$\Man{2}_\e :=\{(x_1+\e ,x_1-\e ,x_3),\, x_1>0,\, x_3 \in \R\} \cup 
\{(x_1-\e ,x_1+\e ,x_3),\, x_1<0,\, x_3 \in \R\} \, ,$$
and $\Man{3}_\e= \R^3 \setminus\left(\Man{1}\cup \Man{2}_\e\right)$.
But this other sequence of stratification $\M_\e$ does not converge either to $\M$ in the sense of
the above definition. This second example is a bit trickier since the limiting $\Man{1}$ is obtained
by merging the two connected components of the $\Man{1}_\e$, a case which is again clearly excluded
by definition \ref{def:conv.strat-art}.

\section{Weak convergence of stratifications and the associated stability result}

The aim of this section is to provide a notion of convergence of stratifications which partially
corrects the defects above and allows to take into account the second above approximation of $\M$
(but not the first one yet). This notion of convergence allows the ``merging'' of different
connected components of $\Man{k}_\e$ but does not permit the emergence of new parts of the
stratification (\ie, no creation of new discontinuities for the equation). On the contrary, it
allows the disappearance of some of them (elimination of discontinuities). We address these
questions in a more complete way later in this chapter.

To do so, we concentrate on the equation in $\R^N \times (0,\Tf)$, the case $t=0$ being treated
analogously. In order to formulate the stability result, a notion of convergence of stratifications
of \cite{AEYW} is changed into the more general following definition.
\index{Stratification!convergence of}
\label{page:strat.conv.deux}
\begin{definition}\label{def:conv.strat}\emph{--- Weak convergence of tangentially flattenable
    stratifications.}\smsp
    We say that a sequence $(\M_\eps)_\eps$ of \TFS of $\R^N \times (0,\Tf)$
    converges to a \TFS $\M$ if: 
    for any $k=1..N+1$, for any $(x,t) \in \Man{k}$, there exists $r>0$ and
            $J\geq 1$ such that, for $\e$ small enough
            \begin{enumerate}
                \item[$(a)$] there exists $J$ connected components
                    $(\Man{k}_{j,\e})_{j=1..J}$ of $\Man{k}_{\e}$ such that $$\Man{k}_\e \cap
                    B((x,t),r) = \cup_j \Man{k}_{j,\e}\cap B((x,t),r)\,;$$
                \item[$(b)$] for any $j=1..J$ and $\eps >0$, there exists a $C^{1,1}$-change of
                    coordinates $\Psi^{x,t}_{j,\eps} : B((x,t),r) \to \R^N \times (0,\Tf)$ such that
                    $\Psi^{x,t}_{j,\eps}(\Man{k} \cap B((x,t),r))= \Man{k}_{j,\e} \cap
                    B((x,t),r)$\,;
                \item[$(c)$] the family of changes of coordinates $\Psi^{x,t}_{j,\eps}$ and their
                    inverses $(\Psi^{x,t}_{j,\eps})^{-1}$ converge in $C^1$ to identity in a
                    neighborhood of $(x,t)$ as $\e\to0$\;.
                \item[$(d)$] For any $l<k$, we have
                $$\Man{l}_\e \cap B((x,t),r) = \emptyset \;.$$
\end{enumerate}

    We denote this convergence by $\M_\eps \toTFSw \M$ where ``$w$'' stands for ``weak'' convergence.    
\end{definition}

In the previous definition, we were using local changes of coordinates which transform globally one
stratification into an other one. Here, on the contrary, the changes act
only on a single connected component $\Man{k}_j$ with no information on their effects on the other parts of the
stratification. As we mentioned above, this formulation allows the merging of several connected
components of the $\Man{k}_\e$. The reader can easily check that it applies without any difficulty
to the second approximation of $\M$ in the example of the previous section.

Clearly, no new part of the stratification can be created in this passage to the limit since,
locally, any connected component of $\Man{k}$ is the limit of one or more connected components of
$\Man{k}_\eps$. On the contrary, some parts of the stratification can disappear, in addition to the
merging of connected components with the same dimension, as shown in the following example in
$\R^N$ (we drop the time for the sake of simplicity):

\begin{figure}[!htp]
    \begin{center}
    \includegraphics[width=0.55\textwidth]{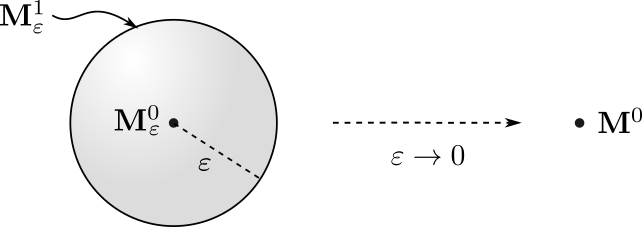}
    \caption{Collapsing of a component.}
    \label{fig:ring}    
    \end{center}
\end{figure}

For $\M$, we take $\Man{0}=\{0\}$ and
$\Man{N}=\R^N\setminus \Man{0}$, and for $\M_\eps$
$$ \Man{0}_\eps=\{0\}\; ,\; \Man{N-1}_\eps=\partial B(0,\eps)\; ,\; \Man{N}_\eps=
\R^N\setminus (\Man{0}_\eps \cup \Man{N-1}_\eps)\; .$$
The reader can easily that $\M_\eps \toTFSw \M$ and $\Man{N-1}$ is empty since $\Man{N-1}_\eps$
vanishes. Here a natural question could be: is it possible to assume that $\Man{0}_\eps=\emptyset$?
We answer to this question in the next sections.

Now, for each $\eps >0$, we consider the associated Hamilton-Jacobi-Bellman problem in the
stratified domain $\M_\eps$. The meaning of sub and supersolutions is the one that
is introduced in Definition~\ref{def:HJBSD}, with the family of Hamiltonians
$\F_\eps$ and $(\F^k_\eps)_k$ that are constructed from some family $\BCL_\eps$.
We write $\HJBS_\e$ for the stratified problem associated to $\F_\eps$ and $(\F^k_\eps)_k$.

In order to formulate the following stability result, we have to define the limiting Hamiltonians:
$\F$---but here this seems classical using the half-relaxed limits method---and the $\F^k$---or some
suitable Hamiltonians---on $\Man{k}$, which, in any case, are defined only if $p\in T_{(x,t)}
\Man{k}_\eps$. The definition of the weak convergence of stratifications gives us a first step in
this direction: with the notations of Definition~\ref{def:conv.strat}, if $(x,t) \in \Man{k}$, we
set
$$ \limiinf\F^k_{j,\eps} (x,t,r,p)= 
\liminf_{\displaystyle{\mathop{\scriptstyle{(\xe,\te)\in \Man{k}_{j,\eps}\to (x,t),\ 
\re \to r}}_{\scriptstyle{\pe \in T_{(\xe,\te)} \Man{k}_{j,\eps}\to p,\ \varepsilon\to 0}}}}\ 
\F^k_{j,\eps}(\xe ,\te,\re,\pe)\; .$$
Notice that this definition is consistent with Definition~\ref{def:conv.strat} since if $\pe \in
T_{(\xe,\te)} \Man{k}_{j,\eps}\to p$ then $p\in T_{(x,t)} \Man{k}$.
\begin{theorem}\label{thm:main.stability}\emph{--- Stability for stratified problems.}\smsp
    Let $\HJBS_\e$ be a sequence of stratified problems in $\R^N \times (0,\Tf)$ associated to
    $\F_\eps$, $\M_\eps$ and $(\F^k_\eps)_k$, such that $\M_\eps \toTFSw \M$. Then the following
    holds: 
  \begin{enumerate} 
       \item[$(i)$] if for all $\eps >0$, $v_\eps$ is a \lsc supersolution of
           $\HJBS_\e$, then $ \underline{v}=\liminf_{*} v_\eps $ is a \lsc
            supersolution of  \HJBS, associated to $\F =\limssup \,\F_\eps$\,;
      \item[$(ii)$] if, for $\eps >0$, $u_\eps$ is an strong \usc subsolution of
          $\HJBS_\e$ and the Hamiltonians $\F_\eps, (\F^k_\eps)_{k=0..N}$ satisfy
          \NCe and \TC with uniform constants on a uniform neighborhood of $\M$, 
          then $\bar{u} = \limsup^*u_\eps$ is a regular \usc subsolution of
          \HJBS associated to $\G^k:=\max_j(\liminf_* \F^k_{j,\eps})$ for any $k=0..N$.
   \end{enumerate}
\end{theorem}

Of course, the ``strong'' convergence of stratification implies the ``weak'' one and therefore
Theorem~\ref{thm:main.stability} a fortiori holds if we replace ``$\M_\eps \toTFSw \M$'' by
``$\M_\eps \toTFSs \M$''.

\medskip

\noindent\textbf{Important ---} In the statement of Theorem~\ref{thm:main.stability}, we have used
the notation $\G^k$ for $\max_j(\limiinf \F^k_{j,\eps})$ because it is not clear a priori that the limit problem
is a consistent stratified problem, \ie that there exists $\BCL$ such that $\F$ is  given by
\eqref{hamil.strat} and $\G^k=\F^k$ is given by \eqref{def:Fk}. We refer to the Section~\ref{suff-cond-stab} for
sufficient conditions that allow to get this property. 

\medskip

\begin{proof} Result $(i)$ is standard since only the $\F_\eps / \F$-inequalities
    are involved and therefore $(i)$ is nothing but the standard stability result
    for discontinuous viscosity solutions with discontinuous Hamiltonians, see
    \cite{I1}. We now focus on getting $(ii)$.

\medskip

    \noindent\textbf{(a)} We assume that $(x_0,t_0) \in \Man{k}$ is a {\em strict} local maximum
    point of $\bar u -\phi$ on $\Man{k}$ where $\phi$ is a $C^1$-function in $\R^N \times (0,\Tf)$
    and we want to show that $$ \G^k(x_0,t_0, \bar u(x_0,t_0), D\phi(x_0,t_0)\leq 0\; .$$
    To do so, it suffices to show that $\liminf_* \F^k_{j,\eps}(x_0,t_0, \bar u(x_0,t_0), D\phi(x_0,t_0)\leq 0$ for any 
    $j$ and we are going to do it with $j=1$ to fix ideas.
    
    On the other hand, since we are going to argue locally, we may assume without loss of generality that
    $\Man{k}= (x_0,t_0)+V_k$ where $V_k$ is a $k$-dimensional subspace of $\R^{N+1}$.

    We consider, in a small neighborhood of $(x_0,t_0)$, 
    $$\chi_\e:(x,t) \mapsto u_\eps (x,t) -\phi (x,t)-L\omega_\e(x,t) \; ,$$ 
    where $L>0$ is a large enough constant to be chosen later on and $$\omega_\e (x,t) = \dist((\Psi^{x,t}_{1,\eps})^{-1}(x,t),\Man{k})\; ,$$ the function
    $\dist(\cdot,\Man{k})$ denoting the distance to $\Man{k}$ which is smooth in a neighborhood of
    $\Man{k}$, except on $\Man{k}$. We point out that we have chosen the change
    $(\Psi^{x,t}_{1,\eps})^{-1}$ of Definition~\ref{def:conv.strat} since our aim is to show the $\limiinf \F^k_{1,\eps}$-inequality.

    \medskip

    \noindent\textbf{(b)} For $\eps>0$ small enough and $L$ large enough, $\chi_\e$ has a maximum point
    $(x_\eps,t_\eps)$ near $(x_0,t_0)$. From the definition of an \TFS, we can find a small
    neighborhood of $(x_0,t_0)$ excluding any point of $\Man{l}$ for $l<k$, and also from
    connected components of $\Man{k}$ than the one of $(x_0,t_0)$ itself. In the same way, the weak convergence of stratification
    also exclude any point of $\Man{l}_\eps $ for $l<k$. So, for $\e>0$ small
    enough, we know that $(x_\eps,t_\eps)\in \Man{l}_\e $ for some $l\geq k$ depending on $\e$.

    We first examine the case when $(\xe,\te)\notin \Man{k}_{1,\e}$. Since $(\Psi^{x,t}_{1,\eps})^{-1}(\xe,\te)$ does not
    belong to $\Man{k}$, $\omega_\e$ is $C^1$ in a neighborhood of $(\xe,\te)$ and $u_\eps$ being a strong \usc
    subsolution of $\HJBS_\e$, we deduce that 
    $$ (\F_\eps)_* \Big(x_\eps,t_\eps, u_\eps(x_\eps,t_\eps),D\phi(x_\eps,t_\eps) +  L
    D\omega_\e (x_\eps,t_\eps) \Big)\leq 0\;.$$

    Next we remark that, on the one hand, $D\big[\dist((x,t),\Man{k})\big]$ is orthogonal to $V_k$
    and on the other hand $\big|D\big[\dist((x,t),\Man{k})\big]\big|=1$ where the distance function
    is differentiable, \ie outside $\Man{k}$. Therefore, by Definition~\ref{def:conv.strat} and the convergence of
    $(\Psi^{x,t}_{1,\eps})^{-1}$ to identity in $C^1$, $D\omega_\e (x_\eps,t_\eps)$ is a transverse vector
    to $\Man{k}$. 
    Moreover, recalling that we are in the flat case, it is easy to see that 
    $$|[D\omega_\e (x_\eps,t_\eps)]^{\bot}|\geq\kappa>0\;,$$
    for some $\kappa\in(0,1)$ which does not depend neither on $\eps$. Here we have
    again strongly used that the distance to $\Man{k}$ is smooth if we are not on $\Man{k}$.

    Using \NCe which  holds in an uniform neighborhood of $\Man{k}$ by assumptions, we deduce
    that the $(\F_\eps)_*$-inequality cannot hold if we chosen $L$ large enough, and of course,
    $L$ can be chosen independently of $\e$ since it depends only on $\kappa$ and $\F$.

    We deduce that, necessarily, $(x_\eps,t_\eps) \in \Man{k}_{1,\e}$ and since $\omega_\e \equiv 0$ on $\Man{k}_{1,\e}$, we have
    $$\F^k_{1,\eps} \Big(x_\eps,t_\eps, u_\eps(x_\eps,t_\eps),D\phi(x_\eps,t_\eps)\Big)\leq 0\; .$$
    But using that $\bar{u} = \limsup^*u_\eps$ and that $(x_0,t_0)$ is a strict local maximum point
    of $\bar u -\phi$ on $\Man{k}$, classical arguments imply that $(x_\eps,t_\eps) \to (x_0,t_0)$ and $u_\eps(x_\eps,t_\eps) \to \bar{u}(x_0,t_0)$
    and the conclusion of the proof follows as in the standard case.
    
    Hence $\limiinf \F^k_{1,\eps}(x_0,t_0, \bar u(x_0,t_0), D\phi(x_0,t_0)\leq 0$ and the index ``$1$'' playing no role, it is true for any $j$ and
    we have $\G^k(x_0,t_0, \bar u(x_0,t_0), D\phi(x_0,t_0)\leq 0$.
    
    It remains to show that $\bar{u}$ is regular on any $\Man{k}$. In fact, $\bar{u}$ is not a strong subsolution but it satisfies $\limiinf \F_\e \leq 0$
    in $\R^N\times (0,\Tf)$. This Hamiltonian is not necessarily equal to $\F_*$ but satisfies \NCe in a neighborhood of $\Man{k}$ for any $k=1..N$.
    Hence $\bar{u}$ is regular by Proposition~\ref{reg-sub}.
\end{proof}

\subsection{Some problematic examples} \label{sec:diffwHamil}

In this section, we look at several example which show both the advantages and disadvantages of
Theorem~\ref{thm:main.stability}, but mainly the defects which have to be corrected. We drop here
the time dependence for the sake of simplicity and investigate the following examples.

\begin{example}\label{prob-example1}
If $(e_1,e_2)$ is the canonical basis of $\R^2$, \ie $e_1=(1,0), e_2=(0,1)$, we
consider the stratification $\M$ defined by
$$ \Man{1}=\R e_1\quad, \quad \Man{2}=\R^2\setminus \Man{1}\; .$$
Next we introduce $\BCL(x_1,x_2)$ defined in the following way: $(b^x,c,l)\in \BCL(x_1,x_2)$ if
$c=1$, $l=1$ and $b^x \in \{e_1\} \times [-1,1]$ if $x_2>0$, $b^x \in \{-e_1\}\times [-1,1]$ if
$x_2<0$. Hence, by the assumptions on $\BCL$, we have 
$$ \BCL(x_1,0)= ([-1,1]\times [-1,1])\times \{1\}\times \{1\}\; ,$$
and, if $p=(p_1,p_2)$, $\F^1$ is given on $\Man{1}$ by
$$ \F^1(x_1,r,p)=\sup_{(b_1,0)\in[-1,1]\times\{0\}}\{-b_1p_1-b_2p_2+r-1\}=|p_1|+r-1\; .$$
    On Figure~\ref{fig:prob-ex-one}, the grey boxes represent the allowed dynamics $b^x$ according to
    the location of $(x_1,x_2)$.
\begin{figure}[!htp]
    \begin{center}
    \includegraphics[width=0.5\textwidth]{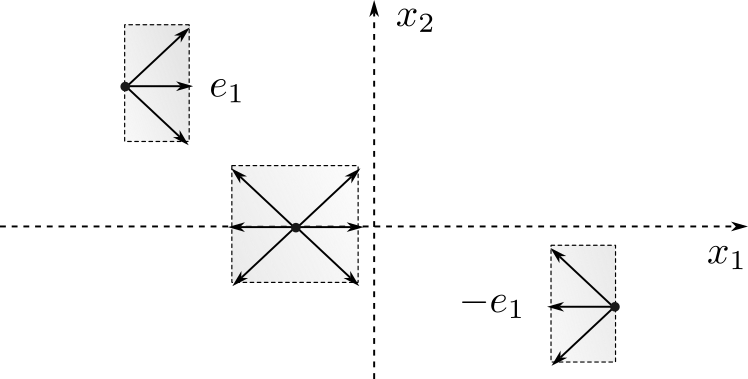}
    \caption{Problematic example one.}
    \label{fig:prob-ex-one}    
    \end{center}
\end{figure}

Now we consider the approximation of $\M$ and $\BCL$ by $\M_\e=\M$ and 
$$ \BCL_\e(x_1,x_2)= (\{(b_x^\e)_1\} \times [-1,1])\times \{1\}\times \{1\}\; ,$$
where, if $\chi:\R\to \R$ is the Lipschitz continuous function given by $ \chi(t)=0$ if $t\leq 0$,
$ \chi(t)=t$ if $0\leq t\leq 1$ and $ \chi(t)=1$ if $t\geq 1$
$$(b_x^\e)_1= \chi(\frac{x_1}{\e}) e_1- \chi(-\frac{x_1}{\e}) e_1\; .$$
Admittedly we are in a continuous setting for $\e >0$ but we ignore this point on purpose, this
regularization yielding a smooth transition between $-e_1$ and $e_1$.

Here, specifically at $x_1=0$ we see that $\F^1_\e(0,r,p)=r-1$, which has nothing to do with $\F^1$,
although it is also clear that 
$$\F^1(x_1,r,p)=\limsup_{y_1\to x_1} \F^1_\e(y_1,r,p)\;.$$
\end{example}

This first example shows that, in general, the $\G^k$ in Theorem~\ref{thm:main.stability} are different
from $\F^k$ and this is a clear problem for the applications. If we want to correct this flaw, we
need to slightly modify the approach we have for this type of convergence.

This is going to be even more striking in the second example. 

\begin{example}\label{prob-example2}
    Here we start from a control problem in $\R^2 \times (0,\Tf)$ where we define the $\BCL$ as
    $$ \BCL(x,t)=\BCL(x):=
    \begin{cases}
    \overline{B(0,1)}\times\{0\}\times \{0\} & \hbox{if $x=0$},\\
    \overline{B(0,1)}\times\{0\}\times \{1\} & \hbox{in $\R^2 \setminus \{0\}$}.
    \end{cases}
    $$
    In other words, we have a fully controllable system ($b$ can be chosen in $\overline{B(0,1)}$),
    $c\equiv0$ and the cost $l$ is $1$ everywhere except at $0$ where
    it is $0$. Hence 
    $$\Man{3}= (\R^2\setminus \{0\})\times (0,\Tf)\quad \hbox{ and }
    \quad\Man{1}=\{0\}\times (0,\Tf)\; .$$

    If we consider the natural approximation obtained by enlarging the discontinuity point
    $$ \BCL_\e(x,t)=\BCL_\e (x):=
    \begin{cases}
    \overline{B(0,1)}\times\{0\}\times \{0\} & \hbox{if $|x|\leq \e$},\\
    \overline{B(0,1)}\times\{0\}\times \{1\} & \hbox{in $\R^2 \setminus \overline{B(0,\e)}$},
    \end{cases}
    $$
    then $\Man{3}_\e= [(\R^2\setminus \overline{B(0,\e)})\cup B(0,\e)] \times (0,\Tf)$ and
    $\Man{2}=\partial B(0,\e)\times (0,\Tf)$. In particular, $\Man{1}_\e=\emptyset$ and
    Theorem~\ref{thm:main.stability} cannot be applied to obtain the $\F^1$-inequality at the limit.
    A defect which has absolutely to be corrected.

    \begin{figure}[!htp]
        \begin{center}
        \includegraphics[width=0.8\textwidth]{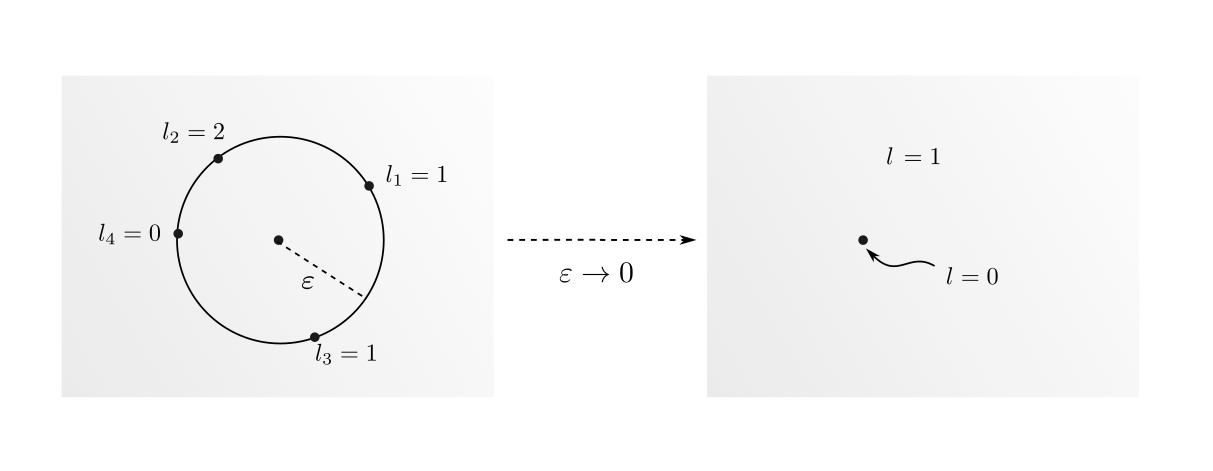}
        \caption{Problematic example two.}
        \label{fig:prob-ex-two}    
        \end{center}
    \end{figure}

    On the other hand, an another approximation shows all the interest of the framework
    of Theorem~\ref{thm:main.stability}, as depicted on Figure~\ref{fig:prob-ex-two}: if we choose
    several distinct elements $e_1, e_2,\cdots,e_J$ of $\R^2$ and
    $$ \BCL_\e(x,t)=\BCL_\e (x):=
    \begin{cases}
    \overline{B(0,1)}\times\{0\}\times \{l_j \} & \hbox{if $x=\e e_j$},\\
    \overline{B(0,1)}\times\{0\}\times \{1\} & \hbox{in $\R^2 \setminus \{\e e_1, \e e_2,\cdots,\e e_J\}$},
    \end{cases}
    $$
    then Theorem~\ref{thm:main.stability} applies and gives $\F^1(x,t,r,p)= p_t -\min_j(l_j)$. 
    Hence
    we recover the expected answer if $\min_j(l_j)=0$.  But we point out that one may have points
    $\e e_j$ with unreasonable cost like $l_j=2$ which the controller should ignore. This is where
    the $\max_j$ in the definition plays an essential role to obtain the right information.
\end{example}

We refer the reader to the end of the chapter to see how to handle these examples.

\subsection{Sufficient conditions for stability}\label{suff-cond-stab}

We conclude this first part devoted to the basic stability results with some sufficient conditions
on $\BCL$ correcting some the above defect and implying a real stability of solutions.

\begin{lemma}  \label{lem:stab.bl.ham} For any $\eps >0$, we assume that $\BCL_\eps$ satisfies
    \hyp{\BCL}, \TCBCL and \NCBCL on a uniform neighborhood of  
    a tangentially flattenable stratification $\M_\e$ with constants independent of $\eps$.
    Moreover, we assume that there exists a tangentially flattenable stratification $\M$ such that $\M_\eps \toTFSs
    \M$.
    \begin{enumerate}
        \item[$(i)$] If the following condition holds
        \begin{equation}\label{hyp:cv-BCL}
            \BCL (x,t) = \limssup_{\e\to0}
        \BCL_\eps (x,t) = \bigcap_{\delta >0}  \bigcap_{\e >0} 
        \left(\overline{K(x,t,\delta,\e)}\right)\; ,
        \end{equation}
        where
        $$K(x,t,\delta,\e):=\bigcup_{\displaystyle{\mathop{\scriptstyle 
            {|(y,s)-(x,t)|\leq \delta}}_{0<\tilde \e \leq \e}}} \BCL_{\tilde \e}(y,s)\; ,$$
        then $\F =\limssup \,\F_\eps$ and the stability result for supersolutions holds.
    \item[$(ii)$] If \eqref{hyp:cv-BCL} holds and if, for any
            $k=0,..,(N+1)$, any $(x,t)\in \Man{k}$ and any $j$
            $$\BCL_\e ([\Psi^{x,t}_{j,\e}]^{-1} (y,s))\to \BCL (y,s)$$
            for any $(y,s) \in B((x,t),r)$ in the sense of the Hausdorff distance  
            where $r,\Psi^{x,t}_{j,\e}$ are as in Definition~\ref{def:conv.strat},
            then Theorem~\ref{thm:main.stability} holds true for subsolutions with $\G^k=\F^k$.  
    \end{enumerate}
\end{lemma}

\begin{proof} We treat successively $(i)$ and $(ii)$.

    \noindent\textbf{The supersolution case ---} 
    If $(b,c,l) \in \BCL (x,t)$, \eqref{hyp:cv-BCL} implies that,
    for all $\delta, \e >0$ small enough, there exists $|(y,s)-(x,t)|\leq \delta$, $0<\tilde \e \leq
    \e$ and $(\tilde b,\tilde c,\tilde l) \in \BCL_{\tilde \e} (y,s)$ such that $|b-\tilde
    b|+|c-\tilde c|+|l-\tilde l|\leq \e$. Therefore, if $|p|+|r|\leq R$ and $|\tilde p-p|+
    |\tilde r-r|\leq 1$,
    $$ -b\cdot p + cr -l \leq -\tilde b\cdot \tilde p + \tilde c \tilde r
    -\tilde l + \e(2R+1)\leq \F_{\tilde \e} (y,s,\tilde r,\tilde p)+ \e(2R+1)\; .$$
    Taking the $\limsup$ in $\delta,\e \to 0$ but also on $\tilde r\to r$, $\tilde p \to p$ and
    using the definition of the $\limssup$, we deduce that
    $$ -b\cdot p + cr -l \leq \limssup \,\F_\eps (x,t,r,p) \; .$$
    Since this is true for any $(b,c,l) \in \BCL (x,t)$, we get that for any $(x,t,r,p)$,
    $$\F(x,t,r,p)\leq \limssup \,\F_\eps (x,t,r,p)\;.$$

    To get the conversely inequality, we consider a sequence 
    $(x_{\tilde \e},t_{\tilde \e},r_{\tilde \e},p_{\tilde \e}) \to (x,t,r,p)$ such that 
    $$ \F_{\tilde \e} (x_{\tilde \e},t_{\tilde \e},r_{\tilde \e},p_{\tilde \e})\to 
    \limssup \,\F_\eps (x,t,r,p)\; .$$
    Since the sets $\BCL_\e$ are compact, there exists 
    $(b_{\tilde \e},c_{\tilde \e},l_{\tilde \e})\in  \BCL_{\tilde \e} (x_{\tilde \e},t_{\tilde \e})$
    such that
    \begin{equation}\label{eqn:becele}
         \F_{\tilde \e} (x_{\tilde \e},t_{\tilde \e},r_{\tilde \e},p_{\tilde \e})=
         -b_{\tilde \e}\cdot p_{\tilde \e}+c_{\tilde \e}r_{\tilde \e}-l_{\tilde \e}\; .
    \end{equation}

    Now we pick $\delta, \e>0$. It is clear that, for $\tilde \e$ small enough, $(b_{\tilde
    \e},c_{\tilde \e},l_{\tilde \e}) \in \overline{K(x,t,\delta,\e)}$. But, since
    $\overline{K(x,t,\delta,\e)}$ is compact, we can assume without loss of generality that 
    $$(b_{\tilde \e},c_{\tilde \e},l_{\tilde \e})\to (b,c,l) \in \overline{K(x,t,\delta,\e)}\; .$$
    This property being true for all $\delta$ and $\e$, we have by assumption $(b,c,l) \in
    \BCL(x,t)$. Letting $\tilde \e \to 0$ in \eqref{eqn:becele}, we get 
    $$ \limssup \,\F_\eps (x,t,r,p)=-b\cdot p+c r-l \leq \F(x,t,r,p)\; ,$$
    which proves that $(i)$ holds: $\F=\limssup \,\F_\eps$.

    \medskip

    \noindent\textbf{The subsolution case ---}
    For the proof of $(ii)$, we only have to examine the convergence of the Hamiltonians $\F^k_\eps$, and 
    not $\limiinf F_\e$. We recall that this is a consequence of the regularity of ``weak subsolutions'' in
    this framework.
    
    Because of the assumptions, we can assume \wlg that we are in a static situation where the
    stratification is fixed and therefore all the Hamiltonians $\F^k_\eps$ are all defined
    on the same set. On the other hand, by \TCBCL, all these Hamiltonians are equicontinuous on
    $\Man{k}=\Man{k}_\eps$ for any $\eps$.  Combining the convergence of $\BCL_\eps$ to $\BCL$ with \NCBCL implies that
    $(\BCL_\eps)|_k$ (the restriction to $\Man{k}\times[0,\Tf]$) converges to $\BCL|_k$. It follows
    directly that
    $$\F_\e^k(x,r,p):=\sup_{\substack{(b,c,l)\in\BCL_\eps(x,t) \\ b\in
          T_x\Man{k}}}
      \big\{-b\cdot p+cr -l \big\}\longrightarrow
      \sup_{\substack{(b,l)\in\BCL(x,t) \\ b\in T_x\Man{k}}}
      \big\{-b\cdot p+cr -l \big\}
    = \F^k(x,r,p)\;.$$
    Combining this pointwise convergence with Ascoli's Theorem, we obtain the local uniform
    convergence of the $\F_\e^k$ to $\F^k$ on $\Man{k}$, and the result is proved.
\end{proof}

\begin{corollary}\label{cor:stability} 
    Under the assumptions of Lemma~\ref{lem:stab.bl.ham}, for any $\e>0$ let $U_\eps$ is the unique
    solution of $\HJBS_\e$. If the functions $U_\eps$ are uniformly bounded, then 
    $$U_\eps\to  U\quad\text{locally uniformly in }\R^N\times[0,\infty)\;,$$ 
    where $U$ is the unique solution of the limit problem \HJBS associated to $(\F^k)_{k=0..N}$.
\end{corollary}

\begin{proof}
    The proof is immediate: by Lemma~\ref{lem:stab.bl.ham},  the half-relaxed limits of the $U_\eps$
    are sub and supersolutions of the limit problem \HJBS thanks to
    Theorem~\ref{thm:main.stability}. Then, the comparison result---Theorem~\ref{comp-strat-RN}---implies 
    that $\limiinf U_\e=\limssup U_\e$, so that all the sequence converges to the common
    limit, $U$, locally uniformly by the classical half-relaxed limits method, see
    Lemma~\ref{ubegalub}.
\end{proof}

\section{Stability under structural modifications of the stratification}

In the previous section, we have provided a stability result in the case when the structure of the
stratification remains unchanged. On the contrary, in this section, we consider cases where this
structure can be changed by the appearance of new discontinuity sets or the disappearance of
existing ones. Anyway, the first stability property shown in the previous section is be the keystone
of this improved result. So, we have to show how to introduce a new part of $\Man{k}$ or remove an
existing one in order to manage these changes of stratifications. Again we only treat the case of
$\R^N\times (0,\Tf)$, the case $t=0$ following similar principles.

It is important to notice that here, such structural modifications of the stratification have an
impact on the associated Hamiltonians and conversely. So, a generalized stability
result necessarily implies considering both at the same time.

\subsection{Introducing new parts of the stratification}\label{sec:INPofS}

The result is the

\begin{proposition}\label{ajout}
    Let $\Sbb=(\Man{k}, \Fk)_k$ be a \SSP and $u:\R^N\times (0,\Tf)\to \R$
    an \usc subsolution of this problem. If $\mathcal{M}$ is a $C^1$-smooth $l$-dimensional
    submanifold of $\Man{k}$ for some $l<k$ and if the normal controllability assumption is satisfied
    in a neighborhood of $\mathcal{M}$, then
    $$ \F^\mathcal{M}(x,t, u, Du)\leq 0\;,$$
    where for $x\in \mathcal{M}$, $t \in (0,\Tf)$, $r\in \R$, $p=(p_x,p_t) \in \R^N\times \R$
    $$ \F^\mathcal{M}(x,t, r ,p):=\sup_{\substack{(b,c,l)\in\BCL(x,t)\\ 
    b \in T_{(x,t)}\mathcal{M}}}\big\{-b\cdot p+cr - l\big\}\;.
    $$

\end{proposition}

This result means that, a priori, we can create an artificial $\Man{l}$-component in $\M$ 
since $\mathcal{M}$ can be seen as some new part of $\Man{l}$. 

But of course, for a concrete use, there are conditions in order that replacing $\Man{l}$ by
$\Man{l}\cup \mathcal{M}$ in $\M$ leads to a new, consistant \SSP:
$\mathcal{M}$ may have a boundary which has to be taken into account, and new viscosity inequalities
have also to be checked on this boundary.

\begin{example}
    $\mathcal{M}=(-1,1) \times \{0\}$ in the whole space $\R^2$ generates a new $\Man{1}$-part but
    also a $\Man{0}$-set with $(\{-1\}\times \{0\})\cup (\{1\} \times \{0\})$.  Moreover, for the
    equation, one also has to examine the $\F^0$-inequalities at these two points.
\end{example}

\begin{proof}
    Since the result is local, we can assume without loss of generality that $\Man{k}=\R^k$ and that
    $\mathcal{M}$ is an affine subspace of $\R^k$. If $\phi : \R^N\times [0,\Tf]\to \R$ is a smooth
    function and $(\xb,\tb) \in  \mathcal{M}$ is a strict, local maximum point of $u-\phi$ on $
    \mathcal{M}$, we have to show that 
    $$ \F^\mathcal{M}(\xb,\tb, u(\xb,\tb), D\phi(\xb,\tb))\leq 0 \; .$$
    To do so, for $0<\e \ll 1$, we consider the function defined for $(x,t)\in\Man{k}=\R^k$
    $$(x,t) \mapsto u(x,t)-\phi(x,t)-\frac{[d(x,t)]^2}{\e}\; ,$$
    where $d(x,t)=d((x,t),\mathcal{M})$ is the distance function to $\mathcal{M}$ which is $C^1$
    outside $\mathcal{M}$ but not on $\mathcal{M}$. On the contrary, $(x,t)\mapsto [d(x,t)]^2$ is
    $C^1$ even on $\mathcal{M}$.

    By standard arguments, this function has a maximum point at $(\xe,\te)$ and
    $$ (\xe,\te)\to (\xb,\tb)\quad u(\xe,\te)\to u(\xb,\tb)\quad\hbox{and}\quad 
    \frac{[d(\xe,\te)]^2}{\e}\to 0\quad \hbox{as  }\e\to 0\; .$$
    Since $u$ is a subsolution of the stratified problem, 
    $$ \F^k\Big(\xe,\te, u(\xe,\te), D\phi(\xe,\te) + 
    \frac{2 d(\xe,\te) Dd(\xe,\te)}{\e}\Big)\leq 0 \; .$$
    In order to deduce the result from this inequality, we use the tangential continuity property on
    $\Man{k}$: if $(\ye,\se)$ is the unique projection of $(\xe,\te)$ on $\mathcal{M}$ (recall that
    locally we are reduced to consider affine subspaces), then
    $|\ye-\xe|+|\te-\se|=d(\xe,\te)$ and
    $$ \F^k\Big(\ye,\se, u(\xe,\te), D\phi(\xe,\te) + 
    \frac{2 d(\xe,\te) Dd(\xe,\te)}{\e}\Big)\leq o_\e (1) \; .$$
    On the other hand, if $b^1 \in T_{(\ye,\se)}\mathcal{M}$ then $b^1\cdot Dd(\xe,\te)=0$
    because $(\ye,\se)$ is the unique projection of $(\xe,\te)$ on $\mathcal{M}$. Therefore,
    restricting the above inequality to such vectors $b^1$, it follows that
    $$\F^{\mathcal{M}}\Big(\ye,\se, u(\xe,\te), D\phi(\xe,\te)\Big)\leq o_\e (1) \; .$$

    In order to conclude, we use again the tangential continuity on $\Man{k}=\R^k$ combined with the
    normal controllability: if $(b,c,l)\in\BCL(\xb,\tb)$ with $b \in T_{(\xb,\tb)}\mathcal{M}$,
    there exists $(b^1_\e,c^1_\e,l^1_\e)\in\BCL(\ye,\se)$ with $b^1_\e \in T_{(\ye,\se)}\mathcal{M}$
    and such that $(b^1_\e,c^1_\e,l^1_\e)\to (b,c,l)$ as $\e \to 0$. Using this property, the result
    is obtained by letting $\e$ tend to $0$.  
\end{proof}

\subsection{Eliminable parts of the stratification}

\index{Stratification!eliminable part of}
In this section, the aim is to remove ``artificial'' parts of the stratification, that is, parts on
which there is no real discontinuity and the viscosity inequalities are just a consequence of those
coming from lower codimensions manifolds. Our result is the
\begin{proposition}\label{elim}
    Let $\Sbb=(\Man{k}, \Fk)_k$ be a \SSP and $u : \R^N\times (0,\Tf)\to
    \R$ an \usc subsolution of this problem. Let $\mathcal{M}\subset \Man{k}$ be a $C^1$-smooth 
    submanifold such that
    \begin{enumerate}
    \item[$(i)$] $\mathcal{M}\subset\overline{\Man{l}}$ for some $l>k$\,;
    \item[$(ii)$] $\mathcal{M}\cup \Man{l}$ is a $l$-dimensional submanifold of $\R^N$\,;
    \item[$(iii)$] $\BCL$ satisfies the tangential continuity assumption on $\mathcal{M}\cup \Man{l}$\;.
    \end{enumerate}
    Then $u$ is a subsolution of
    $$ \tilde \F^l (x,t, u, Du)\leq 0 \quad \hbox{on  } \mathcal{M}\cup \Man{l}\; ,$$
    where, for $x\in \mathcal{M}\cup \Man{l}$, $t \in (0,\Tf)$, $r\in \R$, $p=(p_x,p_t)\in \R^N\times \R$
    $$ \tilde \F^l  (x,t, r ,p):=\sup_{\substack{(b,c,l)\in\BCL(x,t)\\ 
    b\in T_{(x,t)}(\mathcal{M}\cup \Man{l})}}\big\{-b\cdot p+cr - l\big\}\;.$$
\end{proposition}

In other words, this proposition means that $ \Man{l}$ can be replaced by $\mathcal{M}\cup \Man{l}$:
the higher codimension discontinuity manifold $\mathcal{M}$ can be removed and integrated into 
$\Man{l}$. Such a result can be used when a standard continuous HJ-Equation is approximated by a
problem with discontinuities: to recover the right equation at the limit, one has to remove the
artificial discontinuities created by the approximation.

Concerning assumption $(ii)$, notice that of course, including $\mathcal{M}$ into $\Man{l}$ may
completely change its decomposition into connected components: for instance, adding $\{0\}$ to
$\Man{1}=(-\infty;0)\cup(0;+\infty)$ leads to a unique connected component, $\R$ itself.

\begin{example}
    we consider a case similar to the one described at the beginning of the stability chapter: in
    $\R^3$ we define $\M$ by 
    $$ \Man{1}:=\{(0,0,x_3),\, x_3 \in \R\}\; ,\; \Man{2}:=\{(x_1,x_1^2,x_3),\, 
    x_1\in \R \setminus\{0\},\, x_3 \in \R\}\, ,$$
    and $\Man{0}=\emptyset$, $\Man{3}=\R^3 \setminus\left(\Man{1}\cup \Man{2}\right)$.  In this
    setting, it seems relevant to remove $\Man{1}$ and see if we can replace $\Man{2}$ by
    $\{(x_1,x_1^2,x_3),\, x_1\in \R,\, x_3 \in \R\}$. This can be done provided a suitable
    continuity of the Hamiltonian (assumption $(iii)$ above) holds.
\end{example}

\begin{proof}
    Again we can assume without loss of generality that $\Man{l}=\R^l$ and that $\mathcal{M}$ is an
    affine subspace of $\R^l$. If $\phi : \R^N\times [0,\Tf]\to \R$ is a smooth function and
    $(\xb,\tb) \in  \mathcal{M} $ is a strict, local maximum point of $u-\phi$ on $ (\mathcal{M}\cup
    \Man{l})$, we have to show that 
    $$ \tilde \F^l (\xb,\tb, u(\xb,\tb), D\phi(\xb,\tb)) \leq 0\;.$$
    Here the difficulty is that the set $(b,c,l)\in\BCL(x,t)$ with $b \in T_{(x,t)} (\mathcal{M}\cup
    \Man{l})$ is larger than the set for which $b \in T_{(x,t)} \mathcal{M}$.

    If $b \in T_{(x,t)} \mathcal{M}$, the desired inequality is nothing but a consequence of
    the $\Fk$-inequality on $\mathcal{M}$, therefore we can assume \wlg
    that $b \notin T_{(x,t)} \mathcal{M}$. We decompose
    $$ b = b^\top + b^\bot\quad\hbox{with }b^\top \in T_{(x,t)}\mathcal{M},\  
    b^\bot \hbox{  in its orthogonal space}. $$

    For $0<\e\ll1$, we consider on $D=\big\{(x,t) \in \Man{l}=\R^l;\ 
    (x-\xb,t-\tb)\cdot b^\bot >0\big\}$ the function
    $$ (x,t)\mapsto u(x,t)-\phi(x,t) - \frac{\e}{(x-\xb,t-\tb)\cdot b^\bot}\; .$$
    We first remark that the normal controllability assumption on $\Man{k}$ (and therefore on
    $\mathcal{M}$) implies the regularity property 
    $$u(\xb,\tb)= \limsup_{\substack{(x,t)\to(\xb,\tb)\\(x,t)\in D}} u(x,t)\; ,$$
    and because of this property, standard arguments show that this function has a maximum point at
    $(\xe,\te)\in D$ satisfying 
    $$ (\xe,\te)\to (\xb,\tb)\;,\quad u(\xe,\te)\to u(\xb,\tb)\quad\hbox{and}\quad 
    \frac{\e}{(\xe-\xb,\te-\tb)\cdot b^\bot}\to 0\quad \hbox{as }\e\to 0\; .$$

    Using assumption $(iii)$, there exists  
    $(b^1_\e,c^1_\e,l^1_\e)\in\BCL(\xe,\te)$ with $b^1_\e\in T_{(\xe,\te)}\Man{l}$ such that 
    $(b^1_\e,c^1_\e,l^1_\e)\to (b,c,l)$ as $\e \to 0$. The $\F^l$-inequality for such triplet yields
    $$ -b^1_\e\cdot \left (D\phi(\xe,\te) - \frac{\e b^\bot }{((\xe-\xb,\te-\tb)\cdot b^\bot)^2}
    \right )+ c^1_\e u(\xe,\te)- l^1_\e \leq 0\; .$$
    But $(-b^1_\e)\cdot (-b^\bot) \to |b^\bot|^2>0$ as $\e\to 0$ and therefore the corresponding 
    term is positive for $\e$ small enough. We deduce that for such $\e$,
    $$-b^1_\e\cdot D\phi(\xe,\te) + c^1_\e u(\xe,\te)- l^1_\e \leq 0\; ,$$
    and the conclusion follows by letting $\e$ tend to $0$.
\end{proof}

\subsection{Sub and super-stratified problems; generalized stability result}
\index{Stratification!sub and super-stratifications}

The two previous sections lead us to introduce the following definition
\begin{definition}\emph{--- Sub and super stratified problems.}\smsp
    Let $\Sbb = (\Man{k}, \Fk)_k$, $\tilde \Sbb=(\mathbf{\tilde M}^{k}, \tilde \Fk)_k$ 
    be two \SSP associated with the same $\BCL$ set.
    \begin{enumerate}
    \item[$(i)$] $\tilde \Sbb$ is said to be a super-stratified problem of $\Sbb$ if it can be deduced from
        $\M$ by applying a finite (or countable) number of time Proposition~\ref{ajout}.
    \item[$(ii)$] $\tilde \Sbb$ is said to be a sub-stratified problem of $\Sbb$ if it can be deduced from
        $\M$ by applying a finite (or countable) number of time Proposition~\ref{elim}.
    \end{enumerate}
\end{definition}

Before commenting these definitions, we use them to extend the notion of convergence of stratified
problems.
\index{Stratification!convergence of}

\begin{theorem}\label{nsr}\label{ncos}\emph{--- Extended stability result for stratified
    problems.}\smsp
    Let $\Sbb_\e=(\Man{k}_\e,\F^k_\e)_{k,\e}$ be a sequence of standard stratified problems such
    that there exists  $\Sbb=(\Man{k},\F^k)_k$, a sequence 
    $\tilde{\Sbb}_\e= (\tMan{k}_\e,\tilde{\F}^k_\e)_{k,\e}$ and 
     $\tilde{\Sbb}=(\tMan{k},\tilde{\F}^k)_k$ such that
    \begin{enumerate}
        \item[$(i)$] for any $\e>0$, $\tilde{\Sbb}_\e$ is a super-stratified problem of $\Sbb_\e$\,;
        \item[$(ii)$] $\tilde{\M}_\e \toTFSw \tilde{\M}$\,;
        \item[$(iii)$] $\Sbb$ is a sub-stratified problem of $\tilde\Sbb$.
    \end{enumerate}
    Then the stability results of Theorem~\ref{thm:main.stability} remain valid, taking into account
    the addition and removal of subsolution inequalities due to the super/sub stratification induced
    by $\tilde \Sbb$ and $\tilde{\Sbb}_\e$.
\end{theorem}

Theorem~\ref{nsr} makes precise a very simple and natural idea: of course, the conditions imposed by
Theorem~\ref{thm:main.stability} on the convergence of stratified problems are very restrictive and
do not cover (for example) the convergence of problems without discontinuities (like, for instance,
Fillipov's approximation) to a problem with discontinuities. 

To correct this defect, it suffices to introduce suitable ``artificial'' elements of stratification,
using Proposition~\ref{ajout} (thus creating a super-stratified problem) then to use
Theorem~\ref{thm:main.stability} and, at the end, we can drop some useless part of the obtained
stratification using the elimination result of Proposition~\ref{elim}. Of course, all these
operations require suitable tangential continuity or normal controllability assumptions.

\begin{example}
    Denoting by $x=(x_1,x_2)$ the points in $\R^2$, let us consider a stationary discontinuous
    problem along the curve
    $$\Man{1}_\e:=\big\{x_2=\gamma_\e(x_1):=\sqrt{x_1^2+\e^2}:x_2\in\R\big\}\;,$$ 
    the set $\BCL_\e$ being given by:
    \begin{equation}\label{ex.BCL}
        \BCL_\e(x):=\begin{cases}
        B(0,1)\times\{1\}\times\{1\} & \text{if }x_2>\gamma_\e(x_1)\;,\\
        B(0,1)\times\{1\}\times\{0\} & \text{if }x_2<\gamma_\e(x_1)\;,\\
        B(0,1)\times\{1\}\times[0,1] & \text{if }x_2=\gamma_\e(x_1)\;.
        \end{cases}\end{equation}
    The associated Hamiltonians are easy to compute: $\F^2_\e(x,r,p)=r+|p|-\ind{x_2>\gamma_\e(x_1)}$
    and $\F^1_\e(x,r,p')=r+|p'|-1$. We recall that in $\F^1$, $p'$ is the tangential component of
    the gradient.

    The singular point $\{0,0\}$ appears as a specific singularity in the limit stratification $\M$,
    which is given by $\Man{1}=\{x_2=-x_1:x_1<0\}\cup\{x_2=x_1:x_1>0\}$, $\Man{0}=\{(0,0)\}$ and
    $\M^2=\R^2\setminus(\Man{1}\cup\Man{0})$. So, in order to understand the limit as $\e\to0$, we
    creat an artificial singularity $\tMan{0}_\e=\{(0,0)\}$ in $\M_\e$, respecting the structure of
    $\M$. Let also $\tMan{1}_\e:=\{x_2=\gamma_\e(x_2):x_1<0\}\cup
    \{x_2=\gamma_\e(x_1):x_1>0\}$ and $\tMan{2}=\R^2\setminus(\tMan{1}\cup\tMan{0})$.
    The associated set of Hamiltonians is essentially the same, except that there is a new one:
    $\F^0_\e((0,0),r,p)=r-1$. 

    Now, passing to the limit we get
    $$\begin{cases}
        \limssup \F_\e(x,r,p)=r+|p|-\ind{x_2>|x_1|} & \text{in }\R^2\;,\\
        \limiinf \F_\e(x,r,p)=r+|p|-\ind{x_2\geq|x_1|} & \text{in }\R^2\;,\\
        \limiinf \tilde{\F}^2_\e(x,r,p) = r+|p|-\ind{x_2>|x_1|} & \text{in }\tMan{2}\;,\\
        \limiinf \tilde{\F}^1_\e(x,r,p) = r+|p|-1 & \text{on }\tMan{1}\;,\\
        \limiinf \tilde{\F}^0_\e(0,r,p) = r-1 & \text{at }\tMan{0}\;.\\
    \end{cases}$$
        On the other hand, the limit $\BCL$ is given by \eqref{ex.BCL} with $\gamma(x_1):=|x_1|$
        instead of $\gamma_\e$. So, we see that the $\limiinf$ above coincide with the various
        Hamiltonians associated with $\BCL$ and the stability property works. Notice that here we do
        need to perform a substratification (step $(iii)$ in Theorem~\ref{nsr}).
\end{example}


Now we come back to the examples of Section~\ref{sec:diffwHamil} to show how they can be treated.

\begin{example} (solving Example~\ref{prob-example1})\\
    We describe again this example for the reader's convenience: if $(e_1,e_2)$ is the canonical
    basis of $\R^2$, \ie $e_1=(1,0), e_2=(0,1)$, we consider the stratification $\M$ defined by 
    $$\Man{1}=\R e_1\quad, \quad \Man{2}=\R^2\setminus \Man{1}\; .$$
    Next we introduce $\BCL(x_1,x_2)$ defined in the following way: $(b_x,c,l)\in \BCL(x_1,x_2)$ if
    $c=1$, $l=1$ and $b_x \in \{e_1\} \times [-1,1]$ if $x_2>0$, $b_x \in \{-e_1\}\times [-1,1]$ if
    $x_2<0$. Hence, by the assumptions on $\BCL$, we have 
    $$ \BCL(x_1,0)= ([-1,1]\times [-1,1])\times \{1\}\times \{1\}\; ,$$
    and, if $p=(p_1,p_2)$, $\F^1$ is given on $\Man{1}$ by
    $$ \F^1(x_1,r,p)=\sup_{(b_1,0)\in[-1,1]\times\{0\}}\{-b_1p_1-b_2p_2+r-1\}=|p_1|+r-1\; .$$

    Now we consider the approximation of $\M$ and $\BCL$ by $\M_\e=\M$ and 
    $$ \BCL_\e(x_1,x_2)= (\{(b_x^\e)_1\} \times [-1,1])\times \{1\}\times \{1\}\; ,$$
    where, if $\chi:\R\to \R$ is the Lipschitz continuous function given by $ \chi(t)=0$ if $t\leq
    0$, $ \chi(t)=t$ if $0\leq t\leq 1$ and $ \chi(t)=1$ if $t\geq 1$
    $$(b_x^\e)_1= \chi(\frac{x_1}{\e}) e_1- \chi(-\frac{x_1}{\e}) e_1\; .$$
    Admittedly we are in a continuous setting for $\e >0$ but we ignore this point on purpose, this
    regularization yielding a smooth transition between $-e_1$ and $e_1$.

    Here, specifically at $x_1=0$ we see that $\F^1_\e(0,r,p)=r-1$, which has nothing to do with
    $\F^1$, although it is also clear that 
    $$\F^1(x_1,r,p)=\limsup_{y_1\to x_1} \F^1_\e(y_1,r,p)\;.$$

    To correct this flaw, we have to introduce new parts of the stratification at the $\e$-level and
    more precisely we can set 
    \begin{align*}
    \Man{1}_{1,\e}:=\{x_2=0\}  & \ \quad \hbox{with  }\quad \F^1_{1,\e}(0,r,p)=r-1\; ,\\
    \Man{1}_{2,\e}:=\{x_2=\e\} & \ \quad \hbox{with  }\quad \F^1_{2,\e}(0,r,p)=-p_{x_1}+r-1\; ,\\
    \Man{1}_{3,\e}:=\{x_2=-\e\} & \ \quad \hbox{with  }\quad \F^1_{3,\e}(0,r,p)=p_{x_1}+r-1\; .
    \end{align*}
    Applying Theorem~\ref{thm:main.stability} with this super-stratification gives the correct
    $\F^1$ on $\Man{1}$ which turns out to be $\max(\F^1_{1,\e},\F^1_{2,\e},\F^1_{3,\e})$, the three
    Hamiltonians being in fact independent of $\e$.
\end{example}

The same type of argument also gives the answer in the second example. 

\begin{example} (solving Example~\ref{prob-example2})\\
    Again we completely describe this example: we start from a control problem in $\R^2 \times (0,\Tf)$
    where we define the $\BCL$ as $$ \BCL(x,t)=\BCL(x):=
    \begin{cases}
    \overline{B(0,1)}\times\{0\}\times \{0\} & \hbox{if $x=0$},\\
    \overline{B(0,1)}\times\{0\}\times \{1\} & \hbox{in $\R^2 \setminus \{0\}$}.
    \end{cases}
    $$
    In other words, we have a fully controllable system ($b$ can be chosen in $\overline{B(0,1)}$),
    $c\equiv0$ and the cost $l$ is $1$ everywhere except at $0$ where
    it is $0$. Hence 
    $$\Man{3}= \R^2\setminus \{0\}\times (0,\Tf)\quad \hbox{ and }\quad\Man{1}=\{0\}\times (0,\Tf)\; .$$

    Here we can consider several approximations: the first one---and maybe the most natural one---is
    obtained by enlarging the discontinuity point 
    $$ \BCL_\e(x,t)=\BCL_\e (x):=
    \begin{cases}
    \overline{B(0,1)}\times\{0\}\times \{0\} & \hbox{if $|x|\leq \e$},\\
    \overline{B(0,1)}\times\{0\}\times \{1\} & \hbox{in $\R^N \setminus \overline{B(0,\e)}$},
    \end{cases}
    $$ 
    then $\Man{3}_\e= [(\R^2\setminus \overline{B(0,\e)})\cup B(0,\e)] \times (0,\Tf)$ and
    $\Man{2}_\e=\partial B(0,\e)\times (0,\Tf)$. Since $\Man{1}_\e=\emptyset$,
    Theorem~\ref{thm:main.stability} cannot be applied to obtain the $\F^1$-inequality at the limit
    but, if we use a super-stratification of $\M_\e$ obtained by introducing $\Man{1}_\e=\{0\}\times
    (0,\Tf)$ and modifying $\Man{3}_\e$ accordingly, Theorem~\ref{thm:main.stability} applies.

    Another approximation can be 
    $$ \BCL_\e(x,t)=\BCL_\e (x):=
    \begin{cases}
    \overline{B(0,1)}\times\{0\}\times \{1\} & \hbox{if $|x|< \e$},\\
    \overline{B(0,1)}\times\{0\}\times \{\varphi \left(\hat x \right)\} & \hbox{if $|x|= \e$},\\
    \overline{B(0,1)}\times\{0\}\times \{1\} & \hbox{in $\R^N \setminus \overline{B(0,\e)}$},
    \end{cases}
    $$
    where $\hat x = x/|x|$ and $\varphi$ is a continuous function such that
    $\min_{|x|=1}\varphi(x)=0$. If $\bar x$ is a point such that $|\bar x|=1$ and $\varphi(\bar
    x)=0$, then, in order to apply Theorem~\ref{thm:main.stability}, one can enlarge the
    $\M_\e$-stratification by introducing
    $$\Man{1}_\e := \{\e \bar x\} \times (0,\Tf)\; ,$$
    and modifying $\Man{N}_\e$ accordingly.
\end{example}

\

These examples show that, in general, some little hacks on the stratifications have to be used in
order to be able to apply Theorem~\ref{thm:main.stability}. This is why a complete theory of the
stability seems very hard to design.

%

\chapter{Applications and Extensions}
\label{chap:appl-strat}

\abstract{Several applications and extensions are presented in this chapter. Some of them being concrete
    examples, some of them being more abstract (like the large time behavior) or can even be seen as an
    extension of the theory (solutions \`a la Barron-Jensen).}

\section{A crystal growth model -- where the stratified formulation is needed}

The following problem concerns a model of $2$-$d$ nucleation in crystal growth phenomenon. In
\cite{GH}, Giga and Hamamuki use concave Hamiltonians but we re-formulate the equations with convex
ones to be in the framework of this book. Moreover, we consider the problem in $\R^N$ instead of
$\R^2$ since this does not create any additional difficulty. 
\index{Applications!crystal growth}

The simplest equation takes the form
\begin{equation}\label{eq:GH}
u_t + |D_x u|=I(x) \quad \hbox{in  }\R^N \times (0,\Tf)
\end{equation}
where the function $I : \R^N \to \R$ is given by
$$ I(x) = \begin{cases} 1 & \hbox{if  }x\neq 0 ,\\
0 & \hbox{if  }x = 0.
\end{cases}
$$
This equation is associated with a bounded, continuous initial data 
\begin{equation}\label{id:GH}
u(x,0)=\u0(x)\quad \hbox{in  }\R^N\;.
\end{equation}

\subsection{Ishii solutions}

Of course, the key difficulty in this problem comes from the discontinuity of $I$. In terms of
classical viscosity solutions' theory, Ishii's definition yields the subsolution condition 
$$ u_t + |Du|\leq I^* (x)=1\quad \hbox{in  }\R^N \times (0,\Tf)\; ,$$
and the important information that $I(0)=0$ completely disappears here. As a consequence, one easily checks that
$u(x,t)=t$ is an Ishii subsolution associated to the initial data $u_0(x)=0$ in $\R^N$.

On the other hand, and formally for the time being, the classical control interpretation of
\eqref{eq:GH} is that the system can evolve at any velocity $b^x$ with $|b^x|\leq 1$, with cost
$l=1$ outside $0$ and $l=0$ at $0$. In the case $\u0=0$, the natural value function is
$U(x,t)= \min(|x|,t)$ in $ \R^N \times [0,\Tf]$ by adopting the strategy to go as quickly as
possible to $x=0$ and then to stay there.

Clearly $u(x,t) > U(x,t)$ if $|x|<t$  although $U$ should be the ``good solution'' and $u$ is a
subsolution. Therefore, we cannot expect any comparison result in this framework. But it is also clear that $u$
is a kind of ``unnatural'' subsolution, due to the fact that Ishii's definition erases the value $0$ of
$I$ at $x=0$ as we saw it above, which is undoubtedly an important information.

\subsection{The stratified formulation}

In this context, the stratified approach could certainly be simplified but let us stick to our
framework: if $t>0$, taking into account the upper semi-continuity and convexity of $\BCL$, we
introduce
$$ \BCL
(x,t)=\BCL(x)=\begin{cases}
    \big\{ \big((b^x,-1),0,1\big);\ |b^x|\leq 1\big\}\; ,& \hbox{if  }x\neq 0\; ,\\[2mm]
\big\{ \big((b^x,-1),0,l\big);\ |b^x|\leq 1,\ 0\leq l \leq 1\big\}\; ,& \hbox{if  }x= 0\; .
\end{cases}$$
And if $t=0$, $\BCL (x,0)$ is the convex hull of $\BCL(x)\cup \big\{ \big((0,0),1,\u0 (x)\big)\big\}$.

The stratification of $\R^N\times (0,\Tf)$ just contains $\Man{1} = \{0\}\times (0,\Tf)$ and
$\Man{N+1}=(\R^N \setminus \{0\})\times (0,\Tf)$ and, since $I(x)=1$ in $\Man{N+1}$,
$$ 
\F^{N+1}(x,t,p)=p_t + |p_x|-1\;.
$$
While, since $b=(b^x,-1)\in T_{(0,t)}\Man{1}$ is equivalent to $b^x=0$, it follows that
$$
\F^1(t,p)= \max_{\substack{(b,c,l)\in\BCL(0)\\ b^x=0} }\left\{p_t -l\right\}=p_t\;.
$$
For $t=0$, we just get the classical initial condition since $b^t\equiv -1$ for any
$(b,c,l)\in\BCL(x)$ and for any $x$.

Therefore, a subsolution\footnote{Here we use the notion of ``strong stratified subsolution'' to have the $\F_* \leq 0$-inequlity at $x=0$.} of the problem is an \usc function $u:\R^N \times [0,\Tf] \to \R$ 
satisfying 
\begin{equation}\label{GH-str1}
    u_t + |D_x u|\leq 1 \quad \hbox{in  }\R^{N} \times (0,\Tf)\; ,
\end{equation}
\begin{equation}\label{GH-str2}
    u_t \leq 0 \quad \hbox{on  }\Man{1} \; ,
\end{equation}
this last subsolution inequality being understood as a $1$-d inequality which is obtained by looking
at maxima of $u(0,t)-\phi(t)$ for smooth functions $\phi$, while the first one is just the classical
Ishii subsolution definition.

A supersolution of the problem is a \lsc function $v:\R^N \times [0,\Tf] \to \R$ which satisfies
\begin{equation}\label{GH-str3}
v_t + |D_x v|\geq I(x) \quad \hbox{in  }\R^N \times (0,\Tf)\; .
\end{equation}

As we developed in the previous chapters, the stratified formulation consists in super-imposing the
right subsolution inequalities on $\Man{1}$, while the supersolution condition is nothing but the
classical Ishii conditions. Finally it is easy to see that the $\F_{init}$-conditions reduce to
\begin{equation}\label{GH-init}
u(x,0) \leq  \u0(x) \leq v(x,0) \quad \hbox{in  }\R^N \; .
\end{equation}

In this framework, several results hold
\begin{theorem}\emph{--- Crystal growth problem.}\label{thm:GH}
    \begin{enumerate}
        \item[$(i)$] A comparison result between stratified sub and supersolutions of
            \eqref{eq:GH}--\eqref{id:GH}, \ie sub and supersolutions which satisfy
            \eqref{GH-str1}--\eqref{GH-str2} and \eqref{GH-str3} respectively, with
            \eqref{GH-init}.
        \item[$(ii)$] There exists a unique stratified solution of
            \eqref{eq:GH}--\eqref{id:GH}, which is given by
            $$U(x,t) = \inf\left\{ \int_0^t I(X(s))ds + \u0(X(t));\ X(0)=x, \ 
            \ |\dot X(s)|\leq 1 \right\}\; .$$
        \item[$(iii)$] This solution is the minimal Ishii viscosity solution.
        \item[$(iv)$] Finally, if $(I_k)_k$ is a sequence of continuous functions such that 
            $$ \limiinf_k I_k(x)=I(x)\quad \hbox{and}\quad \limssup_k I_k(x)=I^*(x)=1\; ,$$
            then the unique (classical) viscosity solutions $u_k$ associated to $I_k$  converges
            locally uniformly to $U$.  
    \end{enumerate}
\end{theorem}

\begin{proof} 
    The proof just consists in applying the result of Chapters~\ref{chap:strat-def},
    \ref{chap:stratcontr} and \ref{NESR}, and therefore in checking the normal controllability  and
    tangential regularity assumptions, which are obvious here. Then, comparison result $(i)$ is just
    a very particular case of Theorem~\ref{comp-strat-RN}, $(ii)$ is obtained by examining carefully
    the value function of the stratified problem. 

    For $(iii)$, it is enough to remark that any Ishii supersolution is a supersolution of the
    stratified problem, as it was done in Corollary~\ref{Strat-eq-Ishii}.

    Finally, $(iv)$ is a straightforward adaptation of Chapter~\ref{NESR}: indeed, there exists a
    sequence $x_k \to 0$ such that $I_k(x_k) \to 0$ and using the stratification $\Man{1}_k =
    \{x_k\}\times (0,\Tf)$ and $\Man{N}_k=\left(\R^N \times (0,\Tf) \right)\setminus \Man{1}_k$,
    Proposition~\ref{ajout} shows that 
    $$(u_k)_t \leq I_k(x_k) \quad \hbox{in  }\Man{1}_k \;. $$
    Using the stability result (Corollary~\ref{cor:stability}) and part $(i)$ of
    Theorem~\ref{thm:GH} lead directly to $(iv)$.  
\end{proof}

In the introduction of the chapter on stability results, we point out that there are situations
where, instead of applying blindly our stability results, some simpler proofs---or proofs in more
general frameworks---can be used. Here, for example, we have made a point to apply
Corollary~\ref{cor:stability} and, to do so,  the functions $I_k$ have to be continuous outside
$\Man{1}_k$. But the reader can easily verify that such continuity is unnecessary, at least as long
as we want to pass to the limit. 

We point out that the above proof can give the convergence of the sequence $(u_k)$ to the unique
stratified solution $U$, even if the functions $I_k$ are discontinuous. Since the Hamilton-Jacobi Equation
satisfied by $u_k$ may have, in general, several solutions (because $I_k$ can have any type of
discontinuities), this result gives the convergence of all the solutions of these equations to $U$.

\begin{remark}
    In \cite{GH}, Giga and Hamamuki tested several notions of solutions for
    \eqref{eq:GH}--\eqref{id:GH} and remarked that most of them were not completely adapted: for the
    notion of $D$ or $\bar D$-solutions, they tried to impose on $\Man{1}$ an Ishii
    subsolution inequality with $I(x)$, not $I^*(x)$. But this was a $\R^N \times (0,\Tf)$-
    inequality, not a $\Man{1}$-one. Although imposing a stronger subsolution condition on $\Man{1}$
    was going in the right direction, this inequality was too strong compared to \eqref{GH-str2}, at
    least the $\bar D$-ones, and they found that the problem has no $\bar D$-solution in general.
    They ended up considering enveloppe solutions, \ie using Result $(iii)$ of Theorem~\ref{thm:GH}.
\end{remark}

\subsection{Generalization}

Of course, the simplest case we study above can be generalized in several ways, even if we wish to
stay in a similar context: it is clear enough that the case when $I$ vanishes at several points
instead of one can be treated  exactly in the same way, just changing $\Man{1}$. A more intriguing
case which is considered in \cite{GH} is when 
$$ I(x) = \begin{cases} 1 & \hbox{if  }x\notin \mathcal{S}\;,\\
0 & \hbox{if  }x \in \mathcal{S}\;,
\end{cases}
$$
for some closed subset $\mathcal{S}$ of $\R^N$.

Giga and Hamamuki aim at treating the case of very general closed subsets $\mathcal{S}$, which does
not seem possible in our framework---though maybe we are missing something here.
A natural assumption for us is the following: there exists a stratification $\tM=(\tMan{k})_k$
of $\R^N$ such that 
$$\tMan{N}=\mathcal{S}^c \cup \mathrm{Int}\,(\mathcal{S})\; ,$$
where $\mathrm{Int}\,(\mathcal{S})$ denotes the interior of $\mathcal{S}$, and
$$\partial \mathcal{S}=\tMan{N-1}\cup \tMan{N-2}\cdots \cup \tMan{0}\; .$$
Once this hypothesis holds, we then set $\Man{k} = \tMan{k-1} \times (0,\Tf)$ for $1 \leq k \leq N+1$.

Clearly this assumption on $\mathcal{S}$ implies that $\partial \mathcal{S}$ has some regularity
properties but, at least, it allows to use all the stratification arguments and therefore all the
above results can be extended thoroughly.

\section{Combustion -- where the stratified formulation may unexpectedly help}\label{SF-H}

In \cite{BR}, motivated by a model of solid combustion in heterogeneous media, Roquejoffre and the
first author studied the time-asymptotic behavior of flame fronts evolving with a periodic
space-dependent normal velocity. By using the ``level-set approach'', the authors introduce an
Eikonal Equation 
\begin{equation}\label{e1.2}
    u_t+R(x)\vert Du\vert=0 \quad \hbox{ in  } \R^N \times (0, +\infty)\; ,
\end{equation}
where, in the most standard case, $R : \R^N \to \R$ is a positive, Lipschitz continuous function.

In \cite{BR}, results on the propagation are given, in particular on the asymptotic velocity but
only in the case of Lipschitz continuous functions $R$. However, an interesting case---which is the
purpose of an entire but formal section in \cite{BR}---concerns the case when $R$ is discontinuous, 
given in $\R^2$ by
$$
    R(x)=R(x_1,x_2) = \begin{cases} M
    & \hbox{if $x_1 \in \Z$} \cr m & \hbox{otherwise,}\end{cases}
$$
where $m, M$ are positive constants. The interesting case is when $m \ll M$ for which we have
``lines with maximal speed''. 

The stratified approach allows to bridge the gap between the formal results in \cite{BR} in the
discontinuous case, and detailed proofs. This section is devoted to expose such content and we hope this
will help the reader be convinced that the classical proofs for the homogenization of
Hamilton-Jacobi Equations extend easily to the discontinuous case provided the right stratified
formulation is used.

\subsection{The level-set approach}
\index{Applications!level-set approach}

We recall that the ``level-set approach'' consists in identifying a moving front $\Gamma_t$ with the
zero-level-set of a solution $u$ of a ``geometric type'' equation, for which one has a
unique viscosity solution, \ie  $\Gamma_t=\{x\in\R^N:u(x,t)=0\}$,

Based on an idea appearing in Barles \cite{Ba-FP} for constant normal velocity, the
``level-set approach'' was first used for numerical computations by Osher and Sethian \cite{OS} who
did these computations for more general normal velocities, in particular curvature dependent ones.
Then Evans and Spruck \cite{ES-I}, Chen, Giga and Goto \cite{CGG} developed the theoretical basis.
We also refer to Souganidis \cite{S-CIME,S-FP} and to \cite{BS-nga} for a complete description of
the ``level-set approach'' but also for applications to the study of moving fronts in
reaction-diffusion equations.

As we already mentioned above, the key idea here is to represent the moving
front $t\mapsto \Gamma_t$ using the level-set, and in general the $0$-level-set, of a continuous
function $u:\R^N \times [0,+\infty)\to \R$, typically a solution of Equation~\eqref{e1.2} or a more general
parabolic equation
\begin{equation}\label{gen-eqn-lsa}
 u_t + F(x,t,D_xu, D^2_{xx}u)=0 \quad \hbox{ in  } \R^N\times(0,\Tf)\; .
\end{equation}
where $F$ satisfies suitable properties. But, in fact, one remarks that a more adapted way of describing things
consists in saying that the ``level-set approach'' actually describes the evolution of a domain
$t\mapsto \Omega_t$, whose boundary is precisely $\Gamma_t$. In combustion, $\Omega_t$
typically represents the ``burnt region'' while $\Gamma_t$ is the flame front and $\R^N \setminus (\Omega_t\cup \Gamma_t)$ is the
``unburnt region''.

The key result of the ``level-set approach'' can be described in the following way: suppose that we can solve
\eqref{gen-eqn-lsa} to gether with any initial data
\begin{equation}
\label{id:e1.2}
 u(x,0)=\u0(x) \quad \hbox{ in  } \R^N\; ,
 \end{equation}
where $\u0 \in C(\R^N)$ represents the front at time $t=0$ in the sense that $\Gamma_0=\{x:\ \u0(x)=0\}$ and,
for example, $\Omega_0=\{x:\ \u0(x)<0\}$ and $\R^N \setminus (\Omega_0\cup \Gamma_0)=\{x:\ \u0(x)>0\}$. 
Then the sets 
$$\Omega_t=\{x:\ u(x,t)<0\}, \ \Gamma_t=\{x:\ u(x,t)=0\}\ \hbox{ and }\ \R^N \setminus
(\Omega_t\cup \Gamma_t)=\{x:\ u(x,t)>0\}$$
are independent of the choice of $\u0$ satisfying the above conditions, but they depend only on
$\Omega_0$, $\Gamma_0$ and $F$. Of course, opposite signs can be chosen for $u_0$ in $\Omega_0$ and
$\R^N \setminus (\Omega_0\cup \Gamma_0)$: a similar result holds and we come back on the effect of this
change later.

Without entering into details, the above result is based on two key properties of the equation: first a comparison
result for bounded continuous sub and supersolutions and then the fact that \eqref{gen-eqn-lsa} is invariant by
change of unknown function $u\to \varphi(u)$, for all $C^1$-change $\varphi$ such that $\varphi'>0$ in $\R$.

Clearly Equation~\eqref{e1.2} satisfies these two conditions when $R$ is a positive, Lipschitz continuous function
since the classical existence and uniqueness theory applies. This allows to define $t \mapsto \Gamma_t$ as
the level-set evolution of $\Gamma_0$ with normal velocity $R$. In addition, the solution $u$ is
given by the control formula 
\begin{equation}
\label{e1.3}
u(x,t)=\inf\Big\{ \u0(\gamma(t)): \gamma(0)=x,\ \vert\dot\gamma (s)\vert\leq R(\gamma(s))\Big\}
\end{equation}
where $\gamma$ is taken among all piecewise $C^1$ curves.

On this example, the role of the choice of the signs of $u_0$ is clear: by Equation~\eqref{e1.2}, $u_t\leq 0$
and therefore, if the burnt region is defined by $\Omega_t=\{x:\ u(x,t)<0\}$, it increases, an expected phenomena.
With the choice of the other sign, the unburnt region would increase, which would be unsatisfactory from the
modelling point of view.

Hence, the choice of the signs of $u_0$ in $\Omega_0$ and $\R^N \setminus (\Omega_0\cup \Gamma_0)$,
to gether with the equation, gives the direction of propagation of the front by implying the
expansion or shrinking of $\Omega_t$, the direction of propagation for $\Gamma_t$ being either
outward or inward to $\Omega_t$ in one or the other case. Such property holds in general for
level-sets equations, even if, for the Mean Curvature Equation, $$ u_t -\Delta u + \frac{\langle
D^2_{xx} u D_x u,D_x u\rangle}{|D_x u|^2}=0 \quad \hbox{ in  } \R^N\times(0,\Tf)\; ,$$ the signs of
$u_0$ are irrelevant.

As we said, the reader will find in \cite{BR} results on this propagation and on the asymptotic
velocity in the case of Lipschitz continuous functions $R$, the discontinuous case being only
considered formally.

\subsection{The stratified formulation}

We extend the discontinuous $\R^2$-framework to a $\R^N$-one by setting 
$$R(x)=R(x',x_N) =R(x') = \begin{cases} M
 & \hbox{if $x' \in \Z^{N-1}$} \cr m & \hbox{otherwise,}\end{cases}$$
where, as usual $x=(x',x_N)$ with $x'\in \R^{N-1}$, addressing the problem through the stratified
formulation. More precisely, we consider the stratification $\R^N \times (0, +\infty) = \Man{2}\cup
\Man{N+1}$ where $\Man{2}= (\Z^{N-1} \times \R)\times (0, +\infty)$,
and $\Man{N+1}$ is its complementary set in $\R^N \times (0, +\infty)$.
Next, let 
$$\BCL(x,t)=\BCL(x) = \begin{cases}
    \big\{((mv,-1),0,0);\ v\in \R^N,\ |v|\leq 1\big\} &\text{if } x \in \Man{N+1}\;,\\[2mm]
    \big\{((Mv,-1),0,0);\ v\in \R^N,\ |v|\leq 1\big\} &\text{if } x\in\Man{2}\;.
\end{cases}$$
Notice that, since $M>m$, $\BCL$ is actually upper semi-continuous on $\Man{2}$. Therefore a
(strong) stratified subsolution $u:\R^N \times (0, +\infty)$ of \eqref{e1.2} is an \usc function with
satisfies 
\begin{align} 
    & u_t+m\vert Du\vert\leq 0 \quad \hbox{ in  } \R^N \times (0, +\infty)\; ,
    \label{str-sub-e1.2-a} \\[2mm]
    & u_t+M\vert Du\vert\leq 0 \quad \hbox{ in  } \Man{2} \times (0, +\infty)\; ,
    \label{str-sub-e1.2-b}
\end{align}
while a stratified supersolution $v:\R^N \times (0, +\infty)$ of \eqref{e1.2} is a \lsc function
satisfying
\begin{equation} \label{str-sub-e1.2-c}
    v_t+R(x) \vert Dv\vert\geq 0 \quad \hbox{ in  } \R^{N} \times (0, +\infty)\; .
\end{equation}

Using results of Section~\ref{sec:compstrat}, one can easily prove the
\begin{theorem}\label{thm:comp.e1.2} 
    For any $\u0 \in C(\R^N)$, problem \eqref{e1.2}-\eqref{id:e1.2} has a unique stratified
    solution given by \eqref{e1.3}. Moreover, a comparison result holds for this problem.
\end{theorem}

We leave the proof to the reader since it comes from a simple checking of the assumptions required
in Section~\ref{sec:compstrat}.

\subsection{Asymptotic analysis}
\index{Applications!homogenization}

The next question concerns the asymptotic velocity when $t\to +\infty$. A classical method
consists in looking first at initial data of the form $\u0(x)=p\cdot x$ for some $p\in \R^N$, in order to
obtain the velocity when the normal direction is $p$.

The classical hyperbolic scaling $(x,t) \to (x/\e,t/\e)$, which preserves velocities, allows to
reduce to finite times the asymptotic behaviour, leading to study the equation satisfied by the
rescaled function $\ue(x,t):= \e u (x/\e,t/\e)$, namely \begin{equation} (\ue)_{t} + R(\xoe)\vert D
\ue \vert = 0\quad \hbox{in }\R^N  \times (0,+\infty)\;.  \label{EqnR} \end{equation} We also notice
that the initial data is unchanged by the scaling, \ie $\ue(x,0)=p\cdot x$.  We can formulate the
result in the following simple form \begin{theorem}\label{thm:asymp.e1.2} The following limit holds
    $$\lim_{\e\to0}\ue(x,t)= p\cdot x - t \Hb(p)$$ where, for $p=(p_1,p_2,\cdots,p_N)$, $\Hb
    (p)=\max ( M\vert p_{N} \vert, m\vert p \vert)$.  \end{theorem} This theorem implies in
    particular that if $|p|=1$, $\Hb(p)$ is the velocity of the front in the direction $p$. Let us
    first remark that, by Theorem~\ref{thm:comp.e1.2}, since $m\leq R(x)\leq M$ in $\R^N$, $$
    p\cdot x-Mt \leq \ue(x,t) \leq p\cdot x-mt\quad \hbox{ in  }\R^{N}\times(0, +\infty)\;.$$
    Therefore $\ue$ is uniformly locally bounded.

Now, in order to prove the convergence result we provide two proofs. The more general consist in
following the method of Lions, Papanicolaou and Varadhan \cite{LPV} together with the perturbed
test-function method of Evans~\cite{E-PTF1,E-PTF2} as in the article of Briani, Tchou and the two
authors \cite{BBCT}. These arguments allows to treat far more general problems but here we can also
provide simplified arguments.

\medskip

\begin{proof}[Proof of Theorem~\ref{thm:asymp.e1.2}: the common ingredients]

    The first step is the \begin{lemma}\label{EP-BR} For any $p\in \R^N$, there exists a
    unique constant $\Hb(p)$ such that the equation \begin{equation}\label{eq:EP} R(x)|p+D_x w| =
    \Hb(p)\quad \hbox{in } \R^N \end{equation} has a bounded, Lipschitz continuous stratified
    solution $w=w(x,p)$. Moreover, $\Hb(p)=\max ( M\vert p_N \vert, m\vert p \vert)$.  \end{lemma}

\begin{proof} This lemma is classical and so is its proof, except that, in our case, $R$ is
    discontinuous but the method remains the same.

    \medskip

    \noindent\textbf{(a)} for $0<\alpha \ll 1$, we consider the equation
    \begin{equation}\label{eq:EPa} R(x)|p+D_x w^\alpha | + \alpha w^\alpha = 0\quad \hbox{in }
    \R^N\; .  \end{equation} Borrowing arguments in Section~\ref{sec:compstrat} and
    Chapter~\ref{NESR}, it is easy to prove that \eqref{eq:EPa} has a unique stratified solution: if
    $R$ is Lipschitz continuous, such result is standard and can easily be obtained by the Perron
    method of Ishii \cite{Is-Per}.  Here we can use an approximation of $R$ by Lipschitz continuous
    functions from above since $R$ is \usc and then to use the stability results of
    Chapter~\ref{NESR}.

    Now, $w^\alpha$ depends only on $x'$ since $R$ depends only on $x'$ and it is
    $\Z^{N-1}$-periodic since $R$ is $\Z^{N-1}$-periodic: indeed, $w^\alpha(x',x_N)$ and
    $w^\alpha(x'+k,x_N+h)$ are solutions of the same equation for any $k\in \Z^{N-1}$ and $h\in \R$
    and therefore they are equal. Hence, for $k=0$, $w^\alpha(x',x_N)=w^\alpha(x' ,x_N+h)$ for any
    $h\in \R$ and for $h=0$, $w^\alpha(x',x_N)=w^\alpha(x'+k,x_N)$ for any $k\in \Z^{N-1}$, proving
    the claims.

    Moreover, thanks again to the comparison results, $-M|p|\leq \alpha w^\alpha(x) \leq -m|p|$ in
    $\R^N$ since $-M|p|/\alpha$ and $-m|p|/\alpha$ are respectively sub and supersolution of
    \eqref{eq:EPa}. Finally, the $w^\alpha$ are equi-Lipschitz continuous since $\alpha w^\alpha$ is
    uniformly bounded and the term $R(x)|p+q |$ is coercive in $q$, uniformly in $x$. We point out
    that, in all the proof, we use extensively the comparison result for stratified solutions of
    \eqref{eq:EPa}.

    \medskip

    \noindent\textbf{(b)} Applying Ascoli's Theorem to the sequence
    $(w^\alpha(\cdot)-w^\alpha(0))_\alpha$ which is equi-Lipschitz continuous and equi-bounded by
    the periodicity of each $w^\alpha$, we can extract a subsequence $(w^{\alpha
    '}(\cdot)-w^{\alpha'}(0))_{\alpha '}$ which converges uniformly in $\R^N$ (by periodicity) to a
    periodic, Lipschitz continuous function $w$. Moreover, we can assume that $\alpha 'w^{\alpha
    '}(0)$ converges to a constant $-\lambda$. By the stability result for stratified solutions, $w$
    is a stratified solution of $$R(x)|p+D_x w| = \lambda\quad \hbox{in }\R^N\; .$$ In order to
    prove that $\lambda $ is unique, we assume by contradiction that there exists a bounded
    stratified solution $w'$ of $$R(x)|p+D_x w'| = \lambda'\quad \hbox{in } \R^N\; ,$$ for some
    different constant $\lambda'$. 

    Since the functions $(x,t)\mapsto w(x)-\lambda t$ and $(x,t)\mapsto w'(x)-\lambda' t$ are
    stratified solutions of the same equation, therefore for any $t>0$ $$ ||(w(x)-\lambda
    t)-(w'(x)-\lambda' t)||_\infty \leq  ||w(x)-w'(x)||_\infty\; ,$$ an inequality which can hold
    for large $t$ only if $\lambda = \lambda'$, proving the claim about the uniqueness of $\lambda$.

    \medskip

    \noindent\textbf{(c)} It remains to show that $\lambda = \Hb(p)$ is given by $\max ( M\vert p_N
    \vert, m\vert p \vert)$. By the Dynamic Programming Principle, we have, for any $\theta>0$ $$
    w(x) = \inf\left\{\int_0^\theta (b(s)\cdot p + \Hb(p))ds + w(X(\theta));\ X(0)=x,\ \dot
    X(s)=b(s)\in B(X(s))\right\}\; ,$$ where $\B (y) = MB(0,1)$ if $y\in \Z^{N-1}$ and $\B (y) =
    mB(0,1)$ if $y\notin \Z^{N-1}$. Here we trust the reader will be able to translate in this
    setting the framework of Chapters~\ref{chap:strat-def} and \ref{chap:stratcontr} without any
    difficulty, even if we have dropped the $b^t$-term since $b^t\equiv -1$.

    In order to compute the infimum in the above formula, there are several choice for $b(s)$.
    First, at any point of $\R^N$, one can choose $|b(s)|=m$ which comes associated to a minimal
    cost $b(s)\cdot p= -m|p|$; if $X(s)\in \Z^{N-1}$, we can choose $b(s)=+/- Me_N$ to stay on
    $\Z^{N-1}$ and then the minimal cost becomes $b(s)\cdot p= -M|p_N|$. The optimal choice, at
    least if $x\in \Man{2}$ is $\min(-m|p|,-M| p_N|)= -\max(m|p|, M|p_N|)$ since, by the above
    choice of $b(s)$, we have $X(s)\in \Man{2}$ if the maximum is $M|p_N|$.  Choosing this strategy
    for $x\in \Man{2}$, we see that for any $\theta>0$, $$ w(x) \leq \theta (-\max(m|p|, M|p_N|) +
    \Hb(p))+  w(X(\theta))$$ and therefore $\Hb(p) \geq \max(m|p|, M|p_N|)$.

    To prove the equality, we examine the two different cases: if the maximum is $m|p|$ and $p\neq
    0$ (the case $p=0$ is obvious and $ \Hb(0)=0$ since $w$ can be taken as a constant), we notice
    that, for any $b(s)$, $b(s)\cdot p \geq -m|p|$. Hence $b(s)\cdot p + \Hb(p) \geq -m|p|+\Hb(p)$
    and if $-m|p|+\Hb(p)\geq \eta >0$, we get for any choice of $b(s)$ $$ \int_0^\theta (b(s)\cdot p
    + \Hb(p))ds\geq \theta \eta,$$ which leads to a clear contradiction with the boundedness of $w$.

    If the maximum is $M|p_N|$, we cannot have $\Hb(p)\geq M|p_N|+\eta \geq m|p| + \eta$ exactly by
    the same argument: either $X(s)\in \Z^{N-1}$ and the minimal cost is $b(s)=- M|p_N|$, while if
    $X(s)\notin \Z^{N-1}$, it is $-m|p|$. In any case, $b(s)\cdot p + \Hb(p) \geq \eta$ and we
    conclude as above.  \end{proof}

We now continue by the\\

\noindent\textbf{A.-- Simplified proof.}

 Because of the very simple form of the
    initial data for $\ue$ and even more, the simple form of the limit of the $\{\ue\}$, there is a
    very quick proof to conclude. Indeed the function $\chi_\e (x,t):= p\cdot x - t \Hb(p) -\e
    w\left(x/\e,p\right)$ is a solution of \eqref{EqnR} and moreover $$ \chi_\e (x,0)-\e
    ||w(\cdot,p)||_\infty \leq \ue(x,0) \leq \chi_\e (x,0)+\e ||w(\cdot,p)||_\infty\; .$$ Therefore,
    using that $\chi_\e + C$ is also a solution for any constant $C$, by the comparison result we
    get $$ \chi_\e (x,t)-\e ||w(\cdot,p)||_\infty \leq \ue(x,t) \leq \chi_\e (x,t)+\e
||w(\cdot,p)||_\infty\; .$$ Taking into account the form of $\chi_\e$ and the boundedness of $w$,
this gives the result.  

\medskip
Now we turn to \\

\noindent\textbf{B.-- A more general proof.}

\medskip

\noindent Now we provide more general arguments, which allow to take care of more general initial
data and limits. Here, proving the convergence of the sequence $\{\ue\}$ relies on the perturbed
test-function method of Evans \cite{E-PTF1,E-PTF2} as in the article of Briani, Tchou and the two
authors \cite{BBCT}.

Let us introduce the half-relaxed limits $\overline u=\limssup u_{\e}$ and $\underline u=\limiinf
u_{\e}$ which are well-defined since the sequence $\{\ue\}$ is locally uniformly bounded. We recall
that for each $p\in\R^N$, Lemma~\ref{EP-BR} provides a unique real denoted by $\bar H(p)$ such that
there exists a solution $w$ of ergodic problem \eqref{eq:EP}.  The key step is the
\begin{lemma}\label{Homo-BR} The functions $\overline u$ and $\underline u$ are respectively
    (classical) viscosity sub and supersolution of \begin{equation}\label{eq:homo} u_t +
        \Hb(Du)=0\quad \hbox{in } \R^N\times (0,+\infty)\; , \end{equation}
        \begin{equation}\label{id:eq:homo} u(x,0)=p\cdot x \quad \hbox{in } \R^N\; .  \end{equation}
\end{lemma}

\begin{proof} We provide the proof only for $\overline u$, the one for $\underline u$ being
    analogous. 

    \medskip

    \noindent\textbf{(a)} Let $\phi :  \R^N\times (0,+\infty) \to \R$ be a smooth test-function and
    let $(\xb,\tb)$ be a strict local maximum point of $\overline u-\phi$. Since we may assume \wlg
    that $(\overline u-\phi)(\xb,\tb)=0$, we know that for $r,h>0$ small enough, $(\overline
    u-\phi)(x,t) \leq 0$ in $Q^{\xb,\tb}_{r,h}$. Moreover, there exists some $\delta=\delta(r,h) >0$
    such that $(\overline u-\phi)(x,t) \leq -2\delta$ on $\partial_p Q^{\xb,\tb}_{r,h}$.

    We want to show that $\phi_t (\xb,\tb) +  \Hb(D\phi (\xb,\tb))\leq 0$ and to do so, we argue by
    contradiction, assuming on the contrary that $\phi_t (\xb,\tb) +  \Hb(D\phi (\xb,\tb)) > 0$. 

    \medskip

    \noindent\textbf{(b)} The first step consists in considering the \emph{perturbed
    test-function} $$\phi_\e (x,t) := \phi(x,t) + \e w\Big(\xoe,D\phi(\xb,\tb)\Big)$$ where $w$ is
    defined in Lemma~\ref{EP-BR}, and to look at $(\phi_\e)_{t} (x,t) + R(\xoe)\vert D\phi_\e (x,t)
    \vert$ in $Q^{\xb,\tb}_{r,h}$.  Formally, using the equation satisfied by
    $w(\cdot,D\phi(\xb,\tb))$, we have $$\begin{aligned} (\phi_\e)_{t} (x,t) &+ R(\xoe)\vert
        D\phi_\e (x,t) \vert = \phi_t (x,t) +R(\xoe)\Big\vert D\phi (x,t) + D_x
        w(\xoe,D\phi(\xb,\tb))\Big\vert\\ & = \phi_t (\xb,\tb) + R(\xoe)\Big\vert D\phi (\xb,\tb) +
    D_x w(\xoe,D\phi(\xb,\tb))\Big\vert + O(r)+O(h)\\ & = \phi_t (\xb,\tb) +  \Hb(D\phi (\xb,\tb))+
    O(r)+O(h), \end{aligned}$$ the terms $O(r),O(h)$ coming from the replacement of $\phi_t (x,t)$
    by $\phi_t (\xb,\tb)$ and of $D\phi (x,t)$ by $D\phi (\xb,\tb)$. Therefore, taking potentially
    $r,h,\delta$ smaller, we have $$ (\phi_\e)_{t} (x,t) + R(\xoe)\vert D\phi_\e (x,t) \vert \geq
    \delta >0\quad \hbox{in  } Q^{\xb,\tb}_{r,h}.$$ The formal computations above can be justified
    by looking carefully at the stratification formulation but such checking does not present any
    difficulty, it only consists in adding the specific tangential inequality on the lines in
    $\Man{2}=\Z^{N-1}\times\R\times(0,+\infty)$.

    \medskip

    \noindent\textbf{(c)} From the first part of this proof, we know that $\phi(x,t) \geq \overline
    u+ 2\delta$ on $\partial_p Q^{\xb,\tb}_{r,h}$. Therefore, by the definition of $\overline u$, it
    follows that for $\e$ small enough $\phi_\e(x,t) \geq \ue+ \delta$ on $\partial_p
    Q^{\xb,\tb}_{r,h}$. 

    Using the local comparison result for stratified solutions, we conclude that for any $\e>0$
    small enough, $\phi_\e(x,t) \geq \ue+ \delta$ in $Q^{\xb,\tb}_{r,h}$. Then, passing to the
    $\limssup$ yields $$\phi(x,t) \geq \overline u+ \delta\quad\text{in } Q^{\xb,\tb}_{r,h}\;,$$
    which contradicts the fact that $(\overline u-\phi)(\xb,\tb)=0$.  Hence we conclude that $\ou$
is indeed a (classical) subsolution of \eqref{eq:homo}--\eqref{id:eq:homo}.
\end{proof}

Notice that in the above proof we use the classical notion of viscosity solutions for the limit
problem \eqref{eq:homo} while we use stratifed solutions for the construction of the perturbed
test-function at level $\e>0$. The proof now follows easily

Since $\chi(x,t):=p\cdot x - t \Hb(p)$
    is an explicit solution of \eqref{eq:homo}-\eqref{id:eq:homo} for which a (classical) comparison
    result holds, we deduce that $$ \overline u(x,t)\leq p\cdot x - t \Hb(p)\leq \underline u(x,t)
    \quad\hbox{in } \R^N\times [0,+\infty)\;.$$ Using the usual arguments of the half-relaxed limits
method \cf Section~\ref{sec:hrl}, we conclude that $u_\e\to u =\ou=\uu$  and the result is proved.
\end{proof}

\section{Large time behavior}

This section enters a little bit more in the description of open problems, that we consider mainly
in Chapter~\ref{chap:strat-ext}. However here we give substantially more information and some
partial result. Before considering the case where discontinuities occur, let us (very) briefly
recall the situation in the simple periodic framework.

\subsection{A short overview of the periodic case}

We consider here Hamilton-Jacobi Equations of the form $$ u_t + H(x,Du)= 0\quad \hbox{in  }\R^N
\times (0,+\infty)\; ,$$ where $H(x,p)$ is convex and coercive in $p$, and periodic in $x$; for
example, let us assume that it is $\Z^N$-periodic in $x$.

In this framework, the expected large time behavior of the solution $u(x,t)$ is an {\em ergodic
behavior}, \ie $$ u(x,t)=\lambda t + \bar \phi (x) + o(1)\quad \hbox{as  }t \to +\infty ,$$ where
$\lambda$, the {\em ergodic constant}, is the unique constant such that the following {\em ergodic
or additive eigenvalue problem} has a periodic solution $\phi $ $$ H(x,D\phi )=-\lambda \quad
\hbox{in  }\R^N \;, $$
and  $\bar \phi$ is one of the solutions of this ergodic problem\footnote{We remark that, if $\phi$ is a solution of 
the ergodic problem then $\phi+C$ is also a solution for any constant $C\in \R$; hence the ergodic problem has
always many solutions and this invariance by the addition of constants is not, in general, the only
cause of non-uniqueness.}.

We refer to \cite{LPV} for the introduction of the ergodic problem which is
nothing but the ``cell problem'' in homogenization and for the proof of the uniqueness of $\lambda$.
Actually we already met this type of ergodic problem in Section~\ref{SF-H}: some of the basic
arguments to study it are given in the proof of Lemma~\ref{EP-BR}, while a concrete use of its
solution is done in the proof of Lemma~\ref{Homo-BR}.

The proof of such ergodic behaviors can be done in two different ways: either by the ``Weak KAM''
theory initiated by Fathi \cite{F1,F2}, based on dynamical systems arguments using control formulas;
or by PDE-type methods.

The first results in this direction were obtained by Namah and Roquejoffre \cite{NR} for equations
with a particular structure. In order to be a little bit more specific on the kind of results they
obtained and the methods they used, we consider the simplest example where we can describe them,
namely the case when the Hamiltonian is given by
$$ H(x,p)=|p|-f(x)\quad \hbox{for $x\in \R^N,$ $p \in \R^N,$}$$
the function $f$ being
periodic, continuous, positive and $K=\{x: f(x)=0\}$ being a non-empty set. In this case, it can
easily be proved that $\lambda=0$ and the large time behavior is obtained by noticing first that
$u(x,t)$ is decreasing in $t$ on $K$; therefore it converges on $K$. On the other hand, the
half-relaxed limits method associated with a comparison result allows to deduce the convergence to a
function $\bar \phi$ on $\R^N \setminus K$; here the key point to have such a comparison result is the
fact that $0$ is (locally) a strict subsolution. We conclude this short description by pointing out
that here, periodicity does not play an important role, nor the convexity of $H$ in $p$
but only the fact that
$$ H(x,p)\geq H(x,0) \quad \hbox{in  }\R^N \times \R^N\; \hbox{and}\; \max_{\R^N}
H(x,0)=0\footnote{We recall that $H(x,0)$ is $\Z^N$-periodic.}\; .$$

Results in the general framework of strictly convex Hamiltonians were then obtained by Fathi
\cite{F1,F2} using the ``Weak KAM'' theory. In fact, they were not generalizations of the Namah and
Roquejoffre results since, as the reader can see on the example we have chosen above, their result
does not require strict convexity.

To give an idea of  the ``Weak KAM'' theory, we first notice that we can assume \wlg that $\lambda=0\,$ by
replacing $u$ by $u(x,t)-\lambda t$ and $H$ by $H+\lambda$. Using such reduction, we then recall that when $H$
is strictly convex, we can introduce the Lagrangian $L$, which is the Legendre-Fenchel transform of $H$, namely
$$
L(x,v)=\sup_{q\in \R^N}\big\{v\cdot q-H(x,q)\big\}\; ,$$ and the solution of the evolution equation
can be written in terms of the Lagrangian as $$ u(x,t)= \inf \left\{ \int_0^t L(\gamma (s), \dot
\gamma(s))ds + u_0(\gamma (0));\ \gamma(t)=x\right\}\; ,$$ where $u_0$ is the initial data of $u$.

Furthermore, the solution of the ergodic
problem satisfies some dynamical programming property like $$ \phi(x)= \inf \left\{ \int_0^t
L(\gamma (s), \dot \gamma(s))ds + \phi(\gamma (0));\ \gamma(t)=x\right\}\text{ for any }t>0\;.  $$
In the ``Weak KAM'' theory, the main point is to identify the large time behavior of geodesics for
the action functional
$$ \int_0^t L(\gamma (s), \dot \gamma(s))ds\; ,$$ the main result being that the large time behavior
of $(\gamma,\dot \gamma)$ is described in terms of the Aubry-Mather set. Here ``large time'' means
both that $t$ is large AND $(t-s)$ is large too: more precisely, if we set
$\tilde \gamma(s) =\gamma (s+t)$ for $-t\leq s\leq 0$, we are interested in the behavior of
$(\tilde \gamma (s), \dot {\tilde \gamma}(s))$ as $s\to -\infty$.

One of the key results of Fathi was to prove that, for $t$ and $t-s$ large enough (for example if $s$ is bounded), the
geodesics satisfy\footnote{we recall that $\lambda=0$, otherwise we would have $-\lambda$ at the
right-hand side here.}$$ H\left (\gamma(s), \frac{\partial L}{\partial v}(\gamma (s), \dot
\gamma(s))\right) \simeq 0\;.  $$ One way or the other, this property is a cornerstone to obtain the
large time behavior of $u$ using the ``Weak KAM'' theory.

This approach was then revisited, simplified, developed in several directions and generalized by
Roquejoffre \cite{Roq}, Davini and Siconolfi \cite{DS}, Fathi and Siconolfi \cite{FS-wk}. Of course,
this short list of references is far from being complete. We refer to the book of Fathi
\cite{fathi2008weak} and his survey \cite{F-sur} for a more satisfactory one.

This general case can also be treated by PDE methods, with slightly more general assumptions than
convexity, which was first done in Barles and Souganidis \cite{BS-LTB}. Here the idea was to show
that $(u_t)_- \to 0$ as $t \to +\infty$. Roughly speaking, the consequence is that the solution
looks like a subsolution of $H=0$ for large $t$: indeed, if $(u_t)_-\to0$, then $u_t \geq -o(1)$
and, using the equation $u_t+H=0$, we deduce that $H\leq o(1)$.

But, if the initial data is a subsolution, the behavior is well-known since $u(x,t)$ is increasing
in $t$. Using this argument together with the ``compactness'' given by the periodic feature of the
problem, one concludes easily. Such compactness is crucial and we refer to \cite{BS-LTB2} for
counter-examples in the case where we still have $ (u_t)_- \to 0$ as $t \to +\infty$ but without
``compactness''.

Before considering the discontinuous case, we point out that both approaches can be extended to
problems with boundary conditions, \cf for example Ishii \cite{H-WK-N} Barles, Mitake and Ishii
\cite{BIM1}.

\subsection{The discontinuous framework}

After this quick overview of the ``continuous'' theory, the question is: what could we expect to be
easily extendable to the ``discontinuous'' case?  Of course, because of the above framework where we
aim at treating discontinuities of $H(x,p)$ in $x$, the natural stratifications of $\R^N \times
\R$ we have to deal with take the form $$\M = \tM \times
\R \quad\text{where}\ \tMan{} \text{ is a stratification of }\R^N\;,$$ hence we are not
looking at general stratifications in $x$ and $t$.

Let us give some ideas for the PDE approach and let people who are more experts than us to have a
look at the  ``Weak KAM'' theory in the discontinuous case\,\footnote{a not so easy task since now
$L$ is discontinuous, at least on an hyperplane...}

\begin{enumerate}
    \item[$(i)$] One point is clear: under suitable assumptions, we do not see any major difficulty to
        extend Namah-Roquejoffre type results in the discontinuous framework: they only rely on
        the half-relaxed limits method and comparison results, both ingredients which are available
        in the stratified case.
    \item[$(ii)$] The Barles-Souganidis approach is more tricky since---even if it is completely
        transparent in the continuous case---the method involves at the same time sub and
        supersolutions properties, \ie $H_*$ and $H^*$, and, in general, a tripling of variables.
        This does not seem very convenient in the discontinuous case where doubling variables is
        already a major difficulty.
    \item[$(iii)$] We believe that the $(u_t)_-$-estimate can be done by rewriting completely the
        Barles-Souganidis proof in terms of a comparison result but this is not completely
        straightforward and far beyond the scope of this book.
\end{enumerate}
Let us check anyway that a relatively easy proof can be done under the stronger
assumption\,\footnote{ We also assume here that $\lambda=0$}

\begin{assumption}{\hyp{asymp}}{Approximate subsolutions}
    For any $\e>0$, there exists a bounded, $C^1$-function  $(\phi_\e)_\e$ such that
    \begin{equation}\label{ass:long.time}
        H^*(x,D\phi_\e (x))\leq \e \quad \hbox{in  }\R^N\;.
    \end{equation}
\end{assumption}
\vspace*{-1em}

This assumption may seem too strong because of the use of $H^*$ (instead of
$H_*$) but it allows to use the arguments as in \cite{BC-re} to obtain the
$(u_t)_-$-estimate through a simple inf-convolution in $t$.
More precisely, the idea is to introduce
$$ v(x,t):=\inf_{s\geq t}\big\{u(x,s)+(s-t)\eta (s)\big\}\; ,$$
for some suitable function $\eta$ such that $\eta(s) \to 0$ as $s\to +\infty$,
and then to use the stratified comparison result. Under the above assumption
and with suitable hypothesis on $H$, this argument yields the estimate
$$u_t (x,t)\geq-\eta(t)\to0\;.$$
We point out that the proof uses a classical inf-convolution in a direction
which is parallel to the discontinuities of $H$, and is therefore not affected
by them. But again, this proof requires the above unnatural assumption to be
useful in this context.

\subsection{An example}

We consider the $1$-d example
$$ u_t + |u_x + p|^2 = V(x) \quad\hbox{in  } \R \times (0,+\infty)\; ,$$
where $p\in \R$ is a parameter and $V(x)=0$ if $x\in \Z$ and $V(x)=1$ if $x\in \R\setminus \Z$.
Applying the above approach, the $1$-d computations are easy
$$ \lambda(p)=
\begin{cases}
0 & \hbox{if  }|p|\leq 1,\\
1-|p|^2 & \hbox{otherwise.}
\end{cases}
$$
Concerning $\phi$, it is given (up to an additive constant) by
\begin{enumerate}
    \item if $|p|\leq 1$, and $x\in [0,1]$,
        $$ \phi(x)= \begin{cases}
            (1-p)x & \hbox{if  }0\leq x \leq \frac{1+p}{2}\;,\\
        (1+p)(1-x) & \hbox{if } \frac{1+p}{2} < x\leq 1\;,\end{cases}$$
        and then this function is extended by periodicity for $x\notin[0,1]$.
    \item if $|p|\geq 1$, $ \phi(x)= 0$ for any $x\in \R$.
\end{enumerate}
For the large time behavior, the following remarks can be made
\begin{enumerate}
    \item[$(i)$] The case $p=0$ is the Namah-Roquejoffre case which can be
        handled without any difficulty.
    \item[$(ii)$] On the other hand, if $p\neq 0$, we are not anymore in the
        Namah-Roquejoffre framework and assumption~\eqref{ass:long.time}
        requires the existence of $C^1$-functions $\phi_\e$ such that
        $$ |\phi_\e '(x)  + p|^2 \leq 1+\e\quad \hbox{on  }
        \R\quad \hbox{and}\quad
        |\phi_\e '(x)+p|\leq\e \quad \hbox{if  }x\in \Z\; . $$
        If $|p| < 1$ it is easy to check that the assumption is
        satisfied. Actually, since the Hamiltonian is independent of $x$ and since $\phi$ has a particular form,
        we can even take a single function obtained by smoothing in a suitable way the solution of the ergodic problem
        which satisfies the inequality with $\e=0$. But,
        for $|p| \geq 1$, the two properties which are required on $\phi_\e$ are
        incompatible with its boundedness.
\item[$(iii)$] If $|p|\geq 1$, one can conclude by the following arguments: let
    $w$ be the unique solution of
    $$ w_t + |w_x + p|^2 = 1 \quad\hbox{in  } \R \times (0,+\infty)\; ,$$
    $$ w(x,0)=u_0(x)\quad\hbox{in  } \R \; ,$$
    where $u_0$ is a continuous, periodic initial data.
    Notice that we have replaced $V(x)$ by $1$ in the equation.

    For this equation, the ergodic problem has exactly the same ergodic constant
    $\lambda(p)=1-|p|^2$ and the same periodic solutions  (the constant functions, this will be
    proved below). Since this equation is now continuous,
    we know that, as $t\to +\infty$,
    $$ w(x,t)=\lambda(p) t + \phi (x) + o(1) \; ,$$
    where $\phi$ is a solution of the ergodic problem and the $o(1)$ is uniform
    on $\R$.

    The equation for $\phi$ reads $|\phi_x+p|^2=|p|^2$ and rewriting it as $2p\phi_x+|\phi_x|^2=0$,
    one proves easily that the periodic function $p\phi$ is decreasing and therefore $p\phi$
    (hence $\phi$) is a constant function. We claim that $\phi =
    \min_{\R}u_0$.

    Indeed, applying the Oleinik-Lax formula to $\tilde w= w-\lambda(p) t $ yields
    $$ \tilde w(x,t) = \min_{y\in \R} \left(u_0(y)+\frac1{4t}\left(x-y-2pt\right)^2\right) \; ,$$
    and therefore $\tilde w(x+2pt,t) \leq u_0(x)$ for any $x$\footnote{The same result can be obtained by
    a careful examination of the pde satisfied by $\tilde w$.}.

    Choosing $x$ such that $u_0(x)=\min_{\R}u_0$ and noticing that, by
    comparison, $\min_{\R}u_0 \leq \tilde w$ in $\R\times (0,+\infty)$, we
    have $ \tilde w(x+2pt,t) \equiv \min_{\R}u_0$. But the uniform
    convergence of $\tilde w$ to $\phi$ on $\R$ immediately yields that
    $\phi=\min_{\R}u_0$.

    To conclude, we remark that, by comparison results
    $$ \min_{\R}u_0 \leq u(x,t) - \lambda(p)t \leq \tilde w (x,t)
    \quad\hbox{in  } \R \times (0,+\infty)\;.$$
    Indeed, the constant $\min_{\R}u_0$ is a subsolution of the equation
    satisfied by $u(x,t) - \lambda(p)t $, which is itself a subsolution
    for the $\tilde w$-equation. The conclusion immediately follows from
    the uniform convergence of $\tilde w$ to $\min_{\R}u_0$.

    This last case mixes (in some sense) ``weak KAM'' arguments and pde ones:
    indeed we point out the important role of the geodesic $\gamma(t)=x+2pt$
    and of the behavior of the different solutions along the geodesic
    (of course we are here in a very simple framework).
    The key point in this case is that these geodesics cross the
    discontinuity in a transversal way, making it irrelevant.
    This is why $u$ and $w$ have the same behavior.
\end{enumerate}

\section{Lower semicontinuous solutions \`a la Barron-Jensen}
\label{sect:SBJ}

The extension of the Barron-Jensen approach to the stratified case requires a change of definition
since it is based on the fact that, when considering equations with a convex Hamiltonian, one can
just look at minimum points when testing both the sub and supersolutions properties. Of course, the
same is true for stratified problems and leads to a new definition.

\subsection{A typical lower semi-continuous eikonal example}
\label{sect:BJ.example}

Before providing precise definitions and a comparison result, we want to examine a key example in
order to recall the difficulties which are solved by the Barron-Jensen approach. We consider the
Eikonal Equation
$$ u_t + |D_x u| = 0 \quad \hbox{in  } \R^N \times (0,\Tf)\; ,$$
with a \lsc initial data 
$$ u(x,0)= g(x) = \begin{cases} 1 & \hbox{if $x\neq 0$}\\
0 & \hbox{otherwise.}
 \end{cases}$$
 Using---at least formally to begin with---the Oleinik-Lax formula, the ``natural solution'' is
 given by $$ u(x,t):= \min_{|y-x|\leq t}(g(y))=\begin{cases} 0 & \hbox{if  }|x|\leq 1\; ,\\
1 & \hbox{if  }|x|> 1\; .\end{cases}
$$
Therefore the solution is discontinuous and the approach via a \SCR is useless in this \lsc
framework.  One wishes to prove, anyway, that $u$ is the unique solution of the above problem.

Ishii's notion of viscosity solution is not well-adapted, in particular for subsolutions: for
example, $w(x,t) = 1$ in $\R^N \times [0,\Tf)$ is an Ishii viscosity subsolution because it
satisfies the equation and $w^*(x,0) \leq g^*(x) = 1$ in $\R^N$.  But we are far from having
the expected inequality $w\leq u$ in $\R^N \times [0,\Tf)$ that a comparison result would give.

On this example, it is clear that the problem comes from the initial data, and more precisely
the way it is taken, since, with the classical viscosity solutions definition, the upper-semicontinuous
enveloppe erases the value $g(0)=0$. Similarly to the difficulties in the stratified framework, the subsolution
inequality has to be reinforced at $t=0$ but here we cannot impose the ``stratification-like'' inequality
$w^*(0,0)\leq 0$ because this inequality for an \usc subsolution is clearly too strong, it is not even satisfied by $u$.

Therefore, if we wish to take into account \lsc initial data, we have to argue only with \lsc
enveloppes and then super-impose subsolution inequalities at $t=0$ in a {\em suitable way} in order
to be sure that this initial data will be seen. Indeed, the lower continuous function
$$ \tilde u(x,t)=  \begin{cases} 1 & \hbox{if $t > 0$}\\
g(x) & \hbox{if $t=0$,}
 \end{cases}$$
is a \lsc subsolution of the problem but $\tilde u$ ``does not see the initial data enough'' since
$$ \tilde u(0,0) < \liminf\{\tilde u(y,t),\ (y,t)\to (0,0)\ \hbox{with  }t>0\}\; .$$

Besides formulating the notions of viscosity sub and supersolution both for \lsc functions (or
their \lsc enveloppes), an important assumption in the Barron-Jensen approach is to avoid such
problems with the initial data, hence the initial regularity assumption \eqref{A-BJ} below.

\subsection{Definition and regularity of subsolutions}

We use below the acronym \SBJ for \emph{Stratified Barron-Jensen} subsolutions, supersolutions and
solutions of the general equation
\begin{equation}\label{eq:SBJ}
    \F\big(x,t,u,(u_t,Du)\big)=0\text{ in }\R^N\times[0,\Tf]\;.
\end{equation}
In order to get a reasonable comparison result for \eqref{eq:SBJ}, we restrict ourselves to the
following set of assumptions

\label{page:HSBJ}
\begin{assumption}{\HSBJ}{Assumptions for the Stratified Barron-Jensen framework.}\\[-1.3cm]
\begin{enumerate}
    \item[$(i)$] The stratification does not depend on time: for any $k=0..N$,
        $$\Man{k+1} = \tMan{k} \times \R\;, $$
        where $(\tMan{k})_k$ is a stratification of $\R^N$.
    \item[$(ii)$] We are given a classical \lsc and bounded initial data $g$, \ie we assume the 
        \lsc sub and supersolutions $u$ and $v$ we are considering satisfy
        $$ u(x,0) \leq g(x) \leq v(x,0) \quad \hbox{in  }\R^N\;.$$
    \item[$(iii)$] Hamiltonian $\F$ is a classical Hamiltonian of the form
        $$ \F\big(x,t,r,(p_t,p_x)\big)= p_t + \tilde \F(x,t,r,p_x)\; ,$$
        and there exists $0<\tTf \leq \Tf$ such that $\tilde \F$ is independent of $t$ if
        $0\leq t\leq \tTf$ and coercive, \ie there exists $\nu >0$ such that
        $$ \tilde \F(x,t,r,p_x) = \tilde \F(x,\tTf,r,p_x)\geq \nu |p_x| -M|r| -M\; ,$$
        for any $x\in \R^N$, $t\in [0,\tTf]$, $r\in \R$ and $p_x\in \R^N$,
        $M$ being the constant appearing in the assumptions for $\BCL$\;.
    \item[$(iv)$] The ``good framework for stratified solutions'' is satisfied.
\end{enumerate}
\end{assumption}

We can now give some precise definitions for the \SBJ problem.
\index{Stratified solutions!\`a la Barron-Jensen}
\begin{definition}\emph{--- Stratified Barron-Jensen sub and supersolutions.}
    \begin{enumerate}
        \item[$(i)$] A locally bounded, \lsc function $v : \R^N\times [0,\Tf[ \to \R$ is a \SBJ
            supersolution of Equation~(\ref{eq:SBJ}) iff it is an Ishii
            supersolution of this equation on $[0,\Tf]$.
        \item[$(ii)$] A locally bounded, \lsc function $u : \R^N\times [0,\Tf[ \to \R$ is a \SBJ
            subsolution of Equation~(\ref{eq:SBJ}) iff\\[2mm]
            $(a)$ it is a Barron-Jensen subsolution of this equation, \ie for any smooth function
            $\varphi$, at any minimum point $(x,t)$ of $u-\varphi$, 
    $$ \F_*\big(x,t,u(x,t) ,(D_t \varphi(x,t) ,D_x\varphi(x,t) )\big) \leq 0\;;$$
    \\
    $(b)$ for any $k=0,...,(N+1)$, for any smooth function $\varphi$,
    at any minimum point $(x,t)$ of $u-\varphi$ on $\Man{k}$,
            $$ \F^k\big(x,t,u(x,t) ,(D_t \varphi(x,t) ,D_x\varphi(x,t)) \big) \leq 0\;;$$
    \item[$(iii)$]In addition, we will say that $u$ is an $\eta$-strict \SBJ subsolution if the various
        ``$\F_*\leq 0$'' or ``$\F^k \leq 0$'' subsolution inequalities are replaced by strict ``$\F_*\leq -\eta <0$''
        ``$\F^k \leq -\eta <0$'' inequalities, $\eta>0$ being independent of $x$ and $t$\footnote{A notion that we will
        use only locally.}. 
    \end{enumerate}
\end{definition}

We point out that, in this definition, the notion of subsolution is in the spirit of \sSSub but we could
as well choose to present a notion of ``weak Barron-Jensen subsolution''. However, we have decided
not to do so since the aim of this section is just to present few ideas for the extension of stratified solutions in the case
of \lsc data and we do not intend to go too far in this direction. Of course, it is not difficult to imagine that, in order
to deal with ``weak Barron-Jensen subsolution'', we have to assume them to be ``regular'', a notion which has to be
properly redefined here and this is the purpose of the next paragraphs.

As we already noticed in the standard stratified framework and throughout this book, comparison
results require some regularity property of the subsolution with respect to the stratification. 
Whether this property follows automatically from the specific structure of the equation or it has
to be imposed, we cannot escape it.

In the standard stratified case, since subsolutions are \usc, the regularity takes the form of a limsup
property, which is also linked to a regularization by sup-convolution in a first step. We refer to
Section~\ref{sect:sup.reg} for the details.

Here, since subsolutions are \lsc, we have to change the strategy by using the inf-convolution tool.
The consequence is also that the regularity property for subsolutions has to be expressed in terms
of liminf. But, as we already noticed in Remark~\ref{rem:reg-sub-weak}-$(ii)$, such property holds
provided the normal controllability assumption is satisfied, which is the case under \HSBJ above.
More precisely, we have the 

\begin{proposition}\label{reg-sub-BJ}\index{Regularity of subsolutions!in the Barron-Jensen
    framework}\emph{--- Regularity of subsolutions.}\smsp 
    Assume that \HSBJ holds and that $u$ is a bounded, \lsc, \SBJ subsolution. Then $u$ is regular: 
    for any $(x,t)\in  \tMan{k} \times (0,\Tf)$ and $0\leq k<N$,
    \begin{equation}\label{reg-kp-BJ}
    u(x,t)=\liminf\{u(y,s)\ ; (y,s) \to (x,t),\ y \in \tMan{k+1}\cup \cdots \cup \tMan{N}\}\; .
    \end{equation}
    Moreover, if $k=N-1$, then locally $\R^N \setminus \tMan{N-1}$ has two connected components
    $(\tMan{N-1})_+$, $(\tMan{N-1})_-$ and the above result is valid imposing to $y$ to be
    either in  $(\tMan{N-1})_+$ or in $(\tMan{N-1})_-$.
 \end{proposition}

Notice that this result does not provide any similar regularity property as $t\to0$. In the standard
case of \usc subsolutions, this is not needed: the fact that $u-v$ is \usc implies that if a
maximizing sequence $(x_k,t_k)$ for $\max(u-v)>0$ is such that $t_k\to0$, using that $\limsup
(u-v)(x_k,t_k)\leq (u-v)(x,0)\leq 0$ easily yields a contradiction. 

On the contrary, if $u$ is \lsc instead of \usc, the argument obviously fails and we have seen in
Section~\ref{sect:BJ.example} above that simple counter-examples to uniqueness can be built because of this.
Hence, in order to get a comparison result, a specific regularity requirement has to be made on the
subsolution as $t\to0$: 
\begin{definition}\label{def:subsol.reg.SBJ}
    A \lsc \SBJ subsolution $u : \R^N\times [0,\Tf[ \to \R$ is initially regular if, for any $x\in\R^N$,
    \begin{equation}\label{A-BJ}
        u(x,0)=\liminf\big\{u(y,t),\ (y,t)\to (x,0)\ \hbox{with  }t>0\big\}\;.
    \end{equation}
\end{definition}

\subsection{The comparison result for stratified Barron-Jensen solutions}

In the \SBJ approach we described above, we are able to present very general results but we just
provide here a uniqueness result using \HSBJ, a framework which slightly generalizes the one of
Ghilli, Rao and Zidani \cite{GRZ}.

\begin{theorem}\label{comp-strat-BJ}\emph{--- Comparison for stratified Barron-Jensen
    solutions.}\smsp
    Assume that \HSBJ holds. Let $u$ and $v$ be two bounded, \lsc, \SBJ sub and supersolution of
    \eqref{eq:SBJ} respectively such that $u$ is initially regular, \ie it satisfies \eqref{A-BJ}.
    Then, the comparison result holds
    $$ u(x,t) \leq v(x,t) \quad \hbox{in  }\R^N \times [0,\Tf)\; .$$
\end{theorem}

\begin{proof} 
    Of course, the approach of Section~\ref{sect:htc} has to be slightly modified. The
     quantities $\max_{\mathcal{K}} (u-v)_+$ and $\max_{\partial_p\mathcal{K}}(u-v)_+$ where
    $\mathcal{K}=\overline{Q^{x,t}_{r,h}[\mF]}$ have to
    be replaced by $\max_{\mathcal{K}}[(u-v)_+]^*$ and $\max_{\partial_p\mathcal{K}}[(u-v)_+]^*$.
    Inded, since $u-v$ is not \usc anymore there is no reason why $(u-v)_+$ would achieve
    its supremum. But, with \adhoc modifications, the ideas of Section~\ref{sect:htc} still apply;
    we skip these modifications here, trusting the reader will be able to perform them.

    We point out anyway, that we face two different situations: with a standard localization
    argument, we can assume that $(u-v)(x,t) \to -\infty$ when $|x|\to +\infty$ or $t\to \Tf$, and
    there exists maximizing sequences $(x_k,t_k)_k$ which are bounded, $t_k$ remaining away from
    $\Tf$; then either, at least along a subsequence, we have $t_k\to t>0$ and an analogue of a \LCR
    is needed, or $t_k\to t=0$ and we face the difficulty connected to the way the initial data is
    assumed and how the initial regularity of $u$ can be used.

    \medskip

    \noindent\textbf{(a)} \emph{The case $t>0$} ---  Here we can argue
    in a similar way to the standard stratified case with the help of the following result,
    which is an easy adaptation Proposition~\ref{reg-by-sc}

    \begin{proposition}\label{reg-by-sc-BJ}
        Under the assumptions of Theorem~\ref{comp-strat-BJ}, if $u$ is a bounded \lsc, stratified
        Barron-Jensen subsolution of \eqref{eq:SBJ}, then for any $(x,t)\in  \tMan{k} \times (0,\Tf)$,
        there exists a sequence of Lipschitz continuous functions $(u^{\e,\alpha})_{\e,\alpha}$ defined in
        a neighborhood $\mathcal{V}$ of $(x,t)$ such that
        \begin{enumerate}
            \item[$(i)$]each $u^{\e,\alpha}$ is a stratified Barron-Jensen subsolutions of
                \eqref{eq:SBJ} in $\mathcal{V}$\,;
            \item[$(ii)$] each $u^{\e,\alpha}$ is semi-concave and $C^1$ on $\tMan{k} \times
                (0,\Tf)$\;,
            \item[$(iii)$] $\sup u^{\e,\alpha} = \lim_{\e,\alpha\to 0}u^{\e,\alpha}= u$ in $\mathcal{V}$.
        \end{enumerate} 
    \end{proposition}

    Proposition~\ref{reg-by-sc-BJ} is proved exactly as Proposition~\ref{reg-by-sc} except that we
    use an inf-convolution instead of a sup-convolution and we treat differently the tangent space
    variable (with the parameter $\e$) and the $t$-variable (with parameter $\alpha$). Of course,
    the regularity of $u$ in terms of liminf, Proposition~\ref{reg-sub-BJ}, is used to proceed here.

    With this adaptation, we get a contradiction in the case $t>0$ exactly as in the standard
    stratified case. 

    \medskip

    \noindent\textbf{(b)} \emph{The case $t=0$} --- By the coercivity assumption on the time interval
    $(0,\tTf)$, $u$ is a Barron-Jensen
    subsolution of the (continuous) equation
    $$ u_t + \nu |D_x u| -M(||u||_\infty+1) = 0\quad \hbox{in  }\R^N \times (0,\Tf)\;.$$
    Therefore, using \eqref{A-BJ}, by the uniqueness property for this problem and the Oleinik-Lax
    (or control) formula, 
    $$ u(x,t) \leq \inf_{|y-x|\leq \nu t}\left(u(y,0) \right) + M(||u||_\infty+1)t\; .$$
    On the other hand, a similar (yet reversed) inequality for $v$ holds, either by the same
    arguments or using the Dynamic Programming Principle
    $$ v(x,t) \geq \inf_{|y-x|\leq M t}\left(u(y,0) \right) - M(||v||_\infty+1)t\; .$$
    From these two inequalities we deduce that if $\delta>0$ is a small constant, 
    $$ u(x,\delta + t) \leq v(x,t) +  K\delta\; ,$$
    for any $\displaystyle 0\leq t \leq \frac{\nu \delta}{M-\nu}$ and for some constant $K$ which can be computed explicitly.

   But the problem is that this inequality is valid for $t$ in a time interval which depends on $\delta$. To get rid
   of this dependence, we remark that, thanks to the assumptions on $\F$ on the time interval $[0,\tTf]$, the function
   $u(x,\delta + t)$ is a subsolution of the problem. On the other hand, because of the Lipschitz continuity of $\F(x,t,r,(p_t,p_x))$ in $r$, $v(x,t)+K\delta\exp(\tilde K t)$ is also a supersolution of the problem for $\tilde K>0$ large enough. By using the
   argument of Step (a) and the fact that $u(x,\delta + t)\leq v(x,t) + K\delta\exp(\tilde K t)$ for $\displaystyle 0\leq t \leq \frac{\nu \delta}{M-\nu}$, we can compare them; hence
$$u(x,\delta + t) \leq v(x,t) +  K\delta\exp(\tilde K t)
\quad \hbox{in  }\R^N \times [0,{\tilde \Tf}-\delta)\; ,$$
for any $0<\delta\ll {\tilde \Tf}$. For fixed $(x,t) \in [0,{\tilde \Tf}/2]$, we can send $\delta$ to $0$ using the lower semi-continuity of $u$: this yields $u\leq v$ in  $\R^N \times [0,{\tilde \Tf}/2]$. And the proof is complete.
\end{proof}

In this book, we have chosen not to develop extensively the stratified analogue of the Barron-Jensen
approach for continuous Hamiltonians, partly because we had to fix some limits to what we decide to
expose, but of course, partly also because it seems difficult to solve some issues.

Among the tractable questions, we point out the case of \emph{continuous subsolutions}: in the
continuous framework, any Ishii subsolution is a Barron-Jensen subsolution and a relatively easy
regularization argument should allow to show that an analogous result holds in the stratified
setting. This argument clearly relies on the use of a suitable notion of ``weak subsolution'', the regularity being
clear from the continuity of the subsolution. In the same way, in the continuous framework, stability results using
only the $\limiinf$ can be proved for the Barron-Jensen approach, and here also such results are probably true.

Finally, the case of obstacle problems with \lsc obstacles $\psi$, \ie
$$ \max\left(\F\big(x,t,u,(u_t,Du)\big); u-\psi\right) =0\text{ in }\R^N\times[0,\Tf]\;,$$
does not enter into the \HSBJ framework since the tangential continuity may not be satisfied. But
notice that the functions $u^{\e,\alpha}$ built in Proposition~\ref{reg-by-sc-BJ} through an
inf-convolution procedure satisfy $u^{\e,\alpha}\leq \psi^{\e,\alpha}\leq \psi$ where the
$\psi^{\e,\alpha}$ are built by using exactly the same procedure. Thus, one may expect that the
results should extend to this more general case under suitable assumptions on the initial data, even
if it is not so clear at once.

Among the less clear issues stands the question whether it is possible to remove the restrictive
assumption on $\F$ near $t=0$ or not. It is worth pointing out that the role of the inf-convolution in the
classical Barron-Jensen argument---typically an inf-convolution in $x$ on the solution, in order to
treat the lower-semi-continuity of the initial data---and the inf-convolution which is used for the
``tangential regularization'', taking care of the stratification, are not completely compatible.
This is what is generating these strong and restrictive assumptions which are not so easily
removable.

In any case, we think that the Barron-Jensen approach can certainly be extended in some of (or all)
the directions we mention above and certainly also in other ones.

\chapter{Further Discussions and Open Problems}
\label{chap:strat-ext}

\abstract{Several possible extensions of the stratified approach are discussed in this chapter:
more general dependence in time, unbounded control problems, etc. Possible applications too:
Large Deviations, homogenization, convergence of numerical
schemes, etc. Finally, the ``unnatural'' Ishii subsolution inequality and the regularity of
value functions are also discussed.}

We start this section by recalling the main ideas of a comparison proof for stratified solutions
\begin{enumerate} 
    \item[{\bf (i)}] We localize, \ie we reduce the proof of a \GCR to the proof of a \LCR.
    \item[{\bf (ii)}] In order to show that the \LCR holds, we first regularize the subsolution by a
            partial sup-convolution procedure\index{Regularization of subsolutions} using the
            tangential continuity and the normal controllability and then (still tangentially) with
        a standard convolution with a smoothing kernel. 
    \item[{\bf (iii)}] After Step \textbf{(ii)} the
            subsolution is Lipschitz continuous \wrt all variables and $C^1$ \wrt the tangent
            variable and we use the ``Magical Lemma'' (Lemma~\ref{lem:comp.fundamental}) to
conclude. \index{Magical Lemma!role in the comparison proof} 
\end{enumerate}

Analyzing these 3 steps in conjunction with Section~\ref{sect:htc} and the examples therein, it
seems rather clear that the localization procedure can be made via various arguments and is not a
limiting step---even if we agree that there are more complicated situations where this might become
a problem. In the same way, Step \textbf{(iii)} is not really a limiting step, especially the way we
use it in the proof by induction.

Hence, in the generalizations we wish to present here, the main issue comes from Step \textbf{(ii)}
and more precisely from the first part, \ie the tangential sup-convolution procedure. This is why we
mainly insist on this point.

However we want to make a remark on Step~\textbf{(i)}. In the proof of Theorem~\ref{comp-strat-RN},
this step is done in the most standard way---explained in Section~\ref{sect:htc}---in order to show
that it can handle several different general situations: both what we call the ``Lipschitz'' and
``convex'' cases in Section~\ref{sect:htc} and, due to the possibility of having $b^t=0$, some
unbounded control case. An other possibility, which requires suitable assumptions on the
Hamiltonians, is to use the localization method of Theorem~\ref{comp:FSP} in order to prove {\em
finite speed of propagation} type results. Such results have, of course, the advantage to take into
account general initial data and solutions without any restriction on their growths at infinity but
they cannot be valid for problems involving unbounded control. Therefore, they can only treat the
cases when $ \F(x,t,r,(p_x,p_t))$ is (or can be written as) $p_t+H(x,t,r,p_x)$, with $H$ possibly
discontinuous in $x,t$ but Lipschitz continuous in $r,p_x$.

\section{More general dependence in time}\label{sect:mgdt}

A quite restrictive---or at least unusual---assumption we have used so far concerns the time
dependence of the Hamiltonians and on the dynamics of the control problems. In general, it is
well-known that a simple continuity assumption is a sufficient requirement.

In stratified problems however, we face two main cases: the general case when the stratification may
depend on time for which space and time play a similar role; and the case when the stratification
does not depend on time. While, in the first case, it seems natural to impose similar assumptions on
$x$ and $t$ for the Hamiltonians, this is no longer the case for the second one and actually this
particular structure allows to weaken the assumptions on the time dependence.

Indeed, in this second case, we can write the stratification as $$\Man{k+1} = \tMan{k} \times \R\;
,$$ where $\tM=(\tMan{k})_k$ is a stratification of $\R^N$ and $\M=(\Man{k})_k$ is the resulting one
in $\R^N \times [0,\Tf]$, which is here presented as the trace on $\R^N \times [0,\Tf]$ of a
stratification on $\R^N \times \R$.

As far as Section~\ref{sect:sup.reg} is concerned, the $t$-variable is always a tangent variable
---this is the main difference with the general case--- and we can use, as it is classical in all
the comparison proofs in viscosity solutions' theory, a ``double parameters sup-convolution''. More
precisely, if $u: \R^N \times [0,\Tf] \to \R$ is a sub-solution, $\tMan{k}$ is identified with
$\R^k$ and $x=(y,z)$ with $y\in \R^k$ and $z\in \R^{N-k}$, we set $$ u^{\e,\beta}(x,t):=\max_{y' \in
\R^k, s \in [0,\Tf]}\Big\{u((y',z),s)- \frac{\left(|y-y'|^2+\e^4\right)^{\alpha/2}}{\e^\alpha} -
\frac{\left(|t-s|^2+\beta^4\right)^{\alpha/2}}{\beta^\alpha} \Big\}, $$ where the parameter $\beta$
governing the regularization in time satisfies $0< \beta \ll \e$.

We drop all the details here but we are sure that they will cause no problem to the reader.

\section{Unbounded control problems}

\index{Control problem!unbounded} In the case of unbounded control problems we face two
difficulties: $(i)$ the localization that we treat---probably in a non-optimal way---in
Section~\ref{sect:htc}, \cf the ``convex case'; $(ii)$ the sup-convolution regularization.

In order to treat this difficulty, we refer the reader to Section~\ref{simple-ex-comp}, in
particular to Theorem~\ref{comp:CC} and Assumption \hyp{BA-HJ-U}. Indeed, in the sup-convolution
procedure, if we examine the proof of Theorem~\ref{reg-by-sc}, we have to manage the error made by
replacing $y$ by $y'$ and this is done by using the dependence in $u$ of the Hamiltonian. This is
exactly what Assumption \hyp{BA-HJ-U} means: performing the Kru\v{z}kov's change of variable $u\to
-\exp(-u)$, one compensates the large terms in ``$D_x H$'' by large terms in ``$D_u H$''.

The same ideas can be used in the stratified framework: we drop the details here since a lot of very
different situations can occur. It would be impossible and maybe useless to try to describe all of
them.

We refer to \cite{Reis2022} and \cite{CRS22} where unbounded control problems are studied in the
hyperplane case under the assumption
$$\lim_{|\alpha|\to+\infty}\frac{l(x,\alpha)}{1+|b(x,\alpha)|}=+\infty\;,$$ locally uniformly in
$x$. This assumption which appears in \cite{Barles90} allows to recover some compactness of
trajectories since fast-moving trajectories get associated with high costs.

\section{Large deviations type problems} \index{Applications!Large Deviations}

The sections of this book in which we consider KPP-type problems give an idea of what can be done in
the context of Large Deviations, but also of the limitations: in the cases where only codimension-1
discontinuities are present, Part~\ref{part:codim1} provides all the needed tools to completely
analyze the problem. We point out that, as it was already remarked in Imbert and Nguyen \cite{IN},
this allows not only to treat in a rather easy way the problem of Bou\'e, Dupuis and Ellis
\cite{GD-BDE} but even to generalize it, by allowing the diffusion matrix to be discontinuous on the
hyperplane, \cf Section~\ref{mg-KPP}.

For more general discontinuities, the situation is not so well understood. Section~\ref{mg-KPP} only
gives few arguments to treat very particular cases.  We can summarize the difficulty in one
sentence: we have learned from the codimension-$1$ case that the vanishing viscosity method
converges to the maximal Ishii subsolution (and solution) of the limiting Hamilton-Jacobi Equation.
Though we think that it is still the case for any type of discontinuities, we are unable to identify
this maximal subsolution, which implies TWO open problems: the identification of the maximal
subsolution and the convergence of the vanishing viscosity method.

Most of Large Deviations problems involve boundary conditions and for these problems, there are two
different cases: either there is no specific difficulty with the boundary conditions (as it is
mainly the case in the four examples presented in \cite{Ba}) and we believe that the above mentioned
tools apply; or there is some specific difficulties with the boundary conditions. In this latter
case, the problem and its solution may not only be related to discontinuities in the Hamiltonians
and/or boundary conditions, see for example \cite{BB1997}.

\section{Homogenization} \index{Applications!homogenization}

We first point out that the arguments which are used in Section~\ref{SF-H}, which are strongly
inspired by those appearing in Barles, Briani, Chasseigne and Tchou \cite{BBCT}, are very flexible:
the identification of the effective Hamiltonian and the application of the perturbed test-function
method of Evans \cite{E-PTF1,E-PTF2} rely on basic results of the theory (existence of solutions,
comparison results and stability). They can therefore be used in a very general framework.

Among all possible applications, the first one we have in mind concerns homogenization in a
chessboard-type configuration, this problem is treated in Forcadel and Rao \cite{MR3262586}. The
approach we describe above together with the results of this book lead to more general results with
simpler proofs; typically the case of all periodic stratified domains can be addressed without
additional difficulties, of course under suitable assumptions. A second one can be found in
Achdou and Le Bris \cite{ALB}: the authors study the homogenization of continuous Hamilton-Jacobi Equations 
where the periodic Hamiltonians are perturbed near the origin; the limiting problem can be identified as
a stratified one where the $\F^0$-subsolution condition at the origin keeps track of the perturbation.

Let us conclude this section with some additional references. Some of them may have been put in the
networks section but we think that they are relevant here since some methods are either similar or
share common features. In addition to \cite{BBCT}, the most specific one on HJ Equations with
discontinuities is Achdou, Oudet and Tchou\cite{zbMATH06647315} for the two-domain case, while in
the networks configurations, the reader can check Achdou and Tchou \cite{zbMATH06448676}, Galise,
Imbert and Monneau \cite{zbMATH06529167}, Forcadel and Salazar \cite{zbMATH07186756}.

\section{Convergence of numerical schemes and estimates} 
\index{Applications!numerical schemes}

For first-order Hamilton-Jacobi Equations, the convergence of numerical schemes is usually obtained
by using the half-relaxed limits method and a comparison result. Therefore we seem to have the key
tools in the stratified framework.

Actually, in \cite{CC}, Cacace and Camilli introduce a semi-Lagrangian approximation scheme for a general
stratified problem and prove the convergence by using these tools, and in particular the comparison result.

The estimates are generally obtained by a comparison result, combined with the consistency of the
scheme and the regularity of the solution. Here the difficulty may come from the different nature of
the equation and the scheme which may appear as being more problematic than in the continuous case.
Maybe the scheme has, in some sense, to ``respect'' the discontinuities and it does not seem so easy
to produce a general theory.

We did not find so many specific references---we apologize if we have missed some works---but the
work of Guerand and Koumaiha \cite{zbMATH07073233} addresses the key difficulties we have in mind.

\section{About Ishii inequalities and weak stratified solutions}

We show in Section~\ref{rweqs} that, roughly speaking, Ishii subsolutions' inequalities are a
consequence of the \LCR for weak stratified solutions. One way or the other, this type of property is
connected to several existing results in the viscosity solutions literature which show the links
between this notion of solutions and monotonicity.

For example, Alvarez, Guichard, Lions \& Morel \cite{AGLM} (see also Biton \cite{Biton}) prove under
suitable assumptions that a monotone semi-group acting on a space of continuous functions is
necessarily the semi-group of viscosity solutions for a possibly fully nonlinear parabolic equation.
In a different framework, the ``geometrical approach to front propagation problems'' of Souganidis
and the first author \cite{BS-nga} allows to define a weak motion of subsets of $\R^N$ which is
almost equivalent to the Level-Set Approach by using: $(i)$ the monotonicity property of sets for
the inclusion relation; $(ii)$ the use of suitable smooth moving sets, which can be seen as the
analogue of test-functions in the geometrical framework.

This second example is closer in the spirit to what is done in Section~\ref{rweqs} and it seems
interesting to re-formulate the idea of Section~\ref{rweqs} in a more general, abstract way, even
if we are going to do so a little bit formally. We consider here a ``stationary'' framework which,
as it is the case in Section~\ref{sect:htc}, is easier to describe, but we trust the reader to be
able to extend the following to the evolution case.

We assume that, for a local equation $\F(x,u,D_x u)=0$ in $\OO$, we are given two ``abstract'' sets
of functions: a set of locally bounded, \usc ``subsolutions'' $\mathcal{S}^{sub}$ and a set of
locally bounded, \lsc  ``supersolutions'' $\mathcal{S}^{sup}$ with the following properties 
\begin{enumerate}
    \item[$(sub)$] — For any $x\in \OO, r>0$ such that $B(x,r)\subset \OO$, any smooth function $\phi$
        in $\OO$ such that $\F_* (x,\phi,D_x \phi)\geq  0$ in $B(x,r)$, we have, for any $u \in
        \mathcal{S}^{sub}$, $$ u(y) -\phi(y) \leq \max_{\partial B(x,r)}\,(u-\phi)\quad \hbox{for
        any }y\in B(x,r) .$$
    \item[$(sup)$] — For any $x\in \OO, r>0$ such that $B(x,r)\subset \OO$, any smooth function $\phi$
        in $\OO$ such that $\F^* (x,\phi,D_x \phi)\leq 0$ in $B(x,r)$, we have, for any $v \in
        \mathcal{S}^{sub}$, $$ \phi(y)-v(y) \leq \max_{\partial B(x,r)}\,(\phi-v)\quad \hbox{for any
        }y\in B(x,r) .$$
\end{enumerate}
We point out that properties $(sub)$ and $(sup)$ can be interpreted as \LCR between either subsolutions
and smooth local supersolutions or supersolutions and smooth local subsolutions. In this context,
it follows that the Ishii inequalities are satisfied. More precisely
\begin{enumerate}
    \item[$(i)$] for any $u \in \mathcal{S}^{sub}$, $\F_*(x,u,D_xu)\leq 0$ in $\OO$ in the viscosity
        sense;
    \item[$(ii)$] for any $v \in \mathcal{S}^{sup}$, $\F^*(x,v,D_x v)\geq 0$ in $\OO$ in the
        viscosity sense.
\end{enumerate}
This result is an easy consequence of the arguments of Section~\ref{rweqs} by looking at {\em
strict} local maxima and minima. Take for instance $u\in\mathcal{S}^{sub}$ and suppose that
$\F_*(x,u,Du)\leq0$ does not hold in the viscosity sense. Then, there exists a test-function $\phi$
such that $u-\phi$ has a strict local maximum at $x$ in $B(x,r)$ and $\F_*(x,\phi,D\phi)>0$. But
using property $(sub)$ above we get a contradiction with the fact that $u-\phi$ has a strict local
maximum at $x$.

\section{Are value functions always regular?}\label{sec:vfar}

What may seem a strange question has an even stranger answer: yes, almost true! This is due to the
lower semi-continuity property, but a little problem still remains: in the stratified framework, we
get \eqref{bonnelim11} but not \eqref{bonnelim12}. In other words, the desired regularity holds on
each $\Man{k}$ for $k<N$ but not on $\Man{N}$ where a ``one-sided'' regularity holds, not a
``two-sided'' one. And this is optimal as shown by the example of the \lsc Heaviside function at $x=0$.

We have anyway the
\begin{lemma}\label{lem:lsc-reg}
    If $u: \R^N \times (0,\Tf)$ is a \lsc function and $\M=(\Man{k})_k$ a stratification of $\R^N
    \times (0,\Tf)$, then $u^*$ is regular, \ie it satisfies \eqref{bonnelim11} on each $\Man{k}$
    ($1 \leq k \leq N$).
\end{lemma}

We have presented this result in the framework of $\R^N \times (0,\Tf)$. But of course, an analogous
one holds for $\R^N\times \{0\}$ and also in the state-constraints framework where it will be very
useful on the boundary since the defect that we cannot provide a ``two-sided'' regularity is
irrelevant there.

\begin{proof}
    We argue by contradiction assuming that, for some $(x,t)\in \Man{k}$,
    $$u^*(x,t) > \limsup \{u^*(y,s): (y,s)\in \Man{k+1}\cup \cdots \Man{N+1}\} \; .$$ 
    By definition of $u^*$, there exists a sequence $(\xe,\te)$ converging to $(x,t)$ such that
    $u^*(x,t)= \lim u(\xe,\te)$ and the above inequality implies that necessarily $(\xe,\te)\in
    \Man{k}$ for $\e$ small enough.

    But, on the other hand, the lower-semicontinuity of $u$ implies the existence of $(\ye,\se) \in
    \Man{k+1}\cup \cdots \Man{N+1}$ such that $\vert (\ye,\se) -(\xe,\te)\vert \leq \e$ and
    $u(\ye,\se) \geq u(\xe,\te)-\e$. Hence $$ \limsup u^* (\ye,\se) \geq \limsup u(\ye,\se) \geq
    u^*(x,t)\; ,$$
    a contradiction which proves the claim.
\end{proof}


\part{State-Constrained Problems}
\label{S-BC}
\fancyhead[CO]{HJ-Equations with Discontinuities: state-constrained problems}


\chapter{Introduction to State-Constrained Problems}
\label{chap:intro.sc}
\abstract{This introduction presents the issues and inherent difficulties of the stratified approach
for state-constrained problems: this approach allows to treat all types of boundary conditions (even
in rather singular settings) in non-smooth domains. The tanker problem is an emblematic example to
this generality. Of course, the boundary creates some difficulties for the regularity of
subsolutions and this is even worse for the initial data.}

In this part we extend the results of Part~\ref{stratRN} to the case of problems set in a bounded or
unbounded domain of $\R^N$ with state-constraints boundary conditions. In
Chapter~\ref{chap:control.tools} on ``Control Tools'', we have already presented finite horizon
control problems in a state-constraints framework; indeed, the space-time trajectory $(X,T)$ has to
satisfy the constraint $T(s) \in [0,\Tf]$ for any $s\geq 0$, \ie $(X(s),T(s))$ has to stay in the
domain $\R^N \times [0,\Tf]$. As a consequence of this general framework, the usual initial data
(the terminal cost) was not given but it has to be computed by solving the $\F_{init}$-equation. It
is therefore natural to investigate problems for which this constraint on $T$ is complemented by a
constraint on $X$, typically $X(s) \in \Omegb$ for some domain $\Omega$~of~$\R^N$.

However we immediately point out that there is a key difference between these two types of
constraints. In Chapter~\ref{chap:control.tools}, when the trajectory reaches the boundary
$\{t=0\}$, it has to stay there because all the dynamics are pointing outward to the domain $\R^N
\times (0,\Tf)$ at $t=0$: for any $x\in\R^N$ and $(b,c,l)\in \BCL(x,0)$, $b^t(x,0)\leq0$. 

On the contrary, here, the normal controllability assumptions
which are going to hold on $\domeg \times (0,\Tf)$ allow both types of dynamics, either pointing
inward or outward the domain at the boundary $\domeg \times (0,\Tf)$. Therefore, the trajectory can
either stay on $\domeg \times (0,\Tf)$ or re-enter the domain. This explains the fundamental
difference between the boundaries $\{t=0\}$ and $\domeg \times (0,\Tf)$ from the trajectory point of
view. 

The ``good news'' is that the points of $\domeg \times (0,\Tf)$ behave essentially as
interior points, and this is why the state-constraints boundary condition does not create much
difficulties to be handled for $t>0$. However, we will need to address some difficulties on the
boundary $\partial\Omega\times\{0\}$, \cf more details in Section~\ref{sect:initial.boundary.interaction}.

\section{Why only state-constrained problems?}

State-constrained problems are a ``natural extension'' of what is done in
Chapter~\ref{chap:control.tools} and Part~\ref{stratRN}, and we already mentioned that this
framework does not lead to major additional difficulties. These two points explain why the study of
such problems is an unavoidable step in the study of stratified problems. However, as the reader may
notice by looking at the \emph{table of contents} of this book, state-constraints boundary
conditions are the only boundary conditions which we study within the stratified framework. This
rises the question: why?

As we are going to explain with more details in the next section, this study readily includes the
classical Dirichlet, Neumann, Robin etc. boundary conditions in a unique framework. But more
importantly, the state-constraints stratified approach allows to deal at the same time with $(i)$
singular (discontinuous) boundary value problems; $(ii)$ non-smooth boundaries; $(iii)$ a mix of
various boundary conditions on different portions of the boundary.

It may be thought that this generality is at the expense of a lot of technicalities. This is not the
case at all---and we were about to write ``on the contrary''---but there are indeed two additional
difficulties:
\begin{enumerate}
    \item[1.] the first one coming from the boundary of the domain and related both to the regularity of
        subsolutions and the possible non-smoothness of the boundary;
    \item[2.] the second one, occurring at $t=0$, concerns the way the initial data is defined and is related
        to the very general framework we want to handle, and the ``good assumptions'' which are
        necessary to do make it work.
\end{enumerate}

These difficulties explain the way Chapter~\ref{chap-StratSC} is organized: as a first step, we are
going to ignore the two above mentioned difficulties and show that state-constrained problems in {\em
stratified domains}\footnote{As we will see it later on, $\Omegb \times (0,\Tf)$ is a stratified
domain if it is a finite union of submanifolds of $\R^N\times (0,\Tf)$.} can easily be handled with
the methods of Part~\ref{stratRN}. Then we show how to address the first difficulty and then the
second one before applying our result to a state-constrained control problem.

We conclude this first part of the introduction by pointing out that other approaches for treating
state-constrained problems in stratified situations appear in Hermosilla and Zidani
\cite{Her-Zid-2015}, Hermosilla, Wolenski and Zidani \cite{Her-Wol-Zid-2017}, Hermosilla, Vinter and
Zidani \cite{Her-Vin-Zid-2017}.

\section{State-constraints and boundary conditions}

Traditionally, Dirichlet, Neumann, Robin, state-constrained problems etc. are considered as separate,
different problems with specific boundary conditions. Even in the control framework, exit
time/stopping time problems or problems with reflections on the boundary seems different; a
combination of them is often delicate to treat. But the stratified formulation of state-constrained
problems allows to treat within the same global framework all these different types of boundary
conditions, both for smooth and non-smooth domains, as well as combinations of them even in rather
singular settings.  All this flexibility comes, on one hand, from the possible discontinuities in
the $\BCL$-sets and therefore on $\F$ and the $\F^k$, \ie both in the equation and boundary
conditions, and, on the other hand, on the possibility to handle ``stratified domains'' which may
have non-smooth boundary.

To convince the reader and to give a more concrete idea of what we mean in the previous paragraph,
we describe in the next section a deterministic control problem proposed  by
P.L.~Lions~\cite{PLL-CF} in one of his lessons at the Coll\`ege de France in 2016---the ``Tanker
problem''---which was one of our main motivation to look at such formulations.

This remark allows us to revisit Dirichlet and Neumann boundary conditions in deterministic control
problems in the next chapter and extend some results to far more general frameworks: discontinuous
Hamiltonians, of course, non-smooth boundary conditions, mixing of boundary conditions and treatment
of rather singular cases (including the above example).

\subsection{A tanker problem mixing boundary conditions}
\label{sect:tanker}

\index{Applications!tanker problem}
\index{Boundary conditions!mixed}
\index{Control problem!tanker problem}

In this situation, a controller has to manage a tanker: the aim is to decide when and where it will
unload its cargo depending typically on the market price for the goods in the cargo. Of course,
this price may depend on the location---typically the country---therefore to the harbor where the
unloading takes place.

In the simplest modelling, the sea is identified with a smooth domain $\Omega \subset \R^2$ and the
harbors are isoled points $P_1, P_2,\cdots, P_L$ on the boundary $\domeg$. The tanker has to be
controlled in such a way that it stays far from the coast and keeps its cargo if prices are low or,
on the contrary, comes to one of the harbors, unloads and sells its cargo when they become higher at
this harbor. The choice of the harbor is clearly part of the problem and there is no reason why
all harbors should be equivalent. Of course, there is an underlying state constraint boundary
condition on $\domeg$ outside $P_1, P_2,\cdots, P_L$ since the tanker cannot accost where no harbor
exists!

In terms of boundary conditions, we are facing a non-standard and rather singular problem involving a
state-constraints boundary condition on $\domeg \setminus \{P_1, P_2,\cdots, P_L\}$ and P.L. Lions
suggested Neumann boundary conditions for the harbors to model the flux of goods which are
sold, leading to a mathematical formulation as follows: 
\begin{equation}\label{PL-Tanker}
\begin{aligned}
    u_t + H(x,t,Du)=0 & \hbox{ in $\Omega \times (0,\Tf)$\; ,} \\[2mm]
    u_t + H(x,t,Du)\geq 0 & \hbox{ on $\domeg\setminus \{P_1, P_2,\cdots, P_L\}
    \times (0,\Tf)$\; ,}\\[2mm]
    \frac{\partial u}{\partial n} =g_i (t)
    & \hbox{ at $P_i$\; , for $i=1,\cdots,L$. }
\end{aligned}
\end{equation}

To the best of our knowledge, there is no work on such type of boundary conditions: here the mixing
of state-constraints and Neumann boundary conditions (which is already not so standard) 
is even more complicated since the Neumann boundary conditions take
place only at isolated points. In fact, even if one can give a sense
to such problems using viscosity solutions' theory, these problems are ill-posed in the sense that
no uniqueness result holds in general, \cf Section~\ref{cex-PLL} for a counter-example. 

The important point is that the Neumann boundary conditions, imposed only at isolated points, are
``not sufficiently seen'' to give sufficient constraints on solutions to provide a uniqueness
result. 

To overcome this difficulty, we use below a re-formulation in terms of stratified
problems, allowing discontinuities in the Hamiltonians as well as in the boundary conditions as we
will develop here. The point is also that the definition of viscosity
solutions for stratified problem consists in ``super-imposing'' some (subsolutions) inequalities on
the discontinuity sets of the Hamiltonians, which can be not only of codimension $1$ but also of
higher codimension. This is exactly what is lacking for obtaining uniqueness, as described in the
previous paragraph.

\subsection{A counter-example for the tanker problem}
\label{cex-PLL}

Let us examine problem \eqref{PL-Tanker} in the following case: $\Omega=\{x_N>0\} \subset \R^N$;
there is only one harbor $P_1=0 \in \domeg$; the equation is given by
$$ u_t +|Du| = 1\quad\hbox{in  }\Omega \times (0,+\infty)\;,$$
and the Neumann boundary condition is 
$$ \frac{\partial u}{\partial n} = g\quad\hbox{at  }0\; \hbox{for all $t\in (0,+\infty)$}\;,$$
for some constant $g\in \R$. For the initial data, we choose $u(x,0)=0$ on $\Omegb$.

To compute a solution, we argue formally: the associated control problem is a problem with a
reflection at $0$ and the controlled trajectory is given by \footnote{We give here a general formula
which the reader will recognize for a reflection term.}
$$ \dot X(s) = \alpha(s)\ds -
\1_{\{X(s)=0\}}n(X(s))\,\mathrm{d}|k|_s\; ,\; X(0) = x \in \Omegb\; ,$$
where $\alpha(\cdot)$ is the control taking values in $B(0,1)$. The term
$-\1_{\{X(s)=0\}}n(X(s))\,d|k|_s$ is the reflection at $0$, $(|k|_s)_s$ being the intensity of the
reflection and $n(X(s))=-e_N$ is the outward unit normal vector to $\domeg$ at $X(s)$. The
value function is $$U(x,t)=\inf_{\alpha(\cdot)}\,\left\{\int_0^t 1\ds +  \int_0^t g
\1_{\{X(s)=0\}}\,\mathrm{d}|k|_s\right\}\; .$$
In this case, the term $\1_{\{X(s)=0\}}\,\mathrm{d}|k|_s$ is nothing but
$\1_{\{X(s)=0\}}\alpha(s)\cdot n(0)\ds$.

\

If $g<0$---a favorable case to unload the cargo---the clear strategy to minimize the cost is to
maximize the integral of $|g| \1_{\{X(s)=0\}}\,\mathrm{d}|k|_s$. Therefore, the strategy is to reach
$0$ as soon as possible and then to have $\alpha(s)\cdot n(0)=1$, \ie $\alpha(s)=n(0)$.
Since $|x|$ is the time which is necessary to reach $0$ from $x$ and then we integrates $g$ till
time $t$, this gives the solution:
$$ U(x,t)= t + g(t-|x|)_+\;,$$

\

Now take $g < g' <0$ and consider $V(x,t)= t + g'(t-|x|)_+$. We claim that $V$ is still a
subsolution of \eqref{PL-Tanker}: indeed, since changing $g$ into $g'$, we just have to check the
inequality at $x=0$, for $t>0$. But, if $(y,s)\sim (0,t)$, then $(s-|y|)_+>0$ and $V(y,s)=s+g'(s-|y|)$. 
Now, since $g'<0$ the super-differential of $V$ is empty at $(0,t)$, leaving us with no
subsolution inequality to check.

Therefore $V$ is a subsolution of the problem but clearly $V>U$ for $t>|x|$ and this shows that no
comparison result can hold.

The interpretation of this counter-example is that the Neumann boundary condition at only one point
(or at isolated points) is not seen enough by the notion of viscosity solution, at least not
sufficiently to imply comparison/uniqueness. This defect will be corrected by the stratified formulation
which superimposes an inequality at $0$ for all~$t$.

\section{A first difficulty: boundary regularity of subsolutions}

We recall that the question of the ``regularity'' of subsolutions is crucial in the stratified
approach: wether it is considered as an assumption for {\em weak} stratified solutions or as a
property for the {\em strong} ones, this regularity is used on each part of the stratification in
the comparison proof and is a key property to make the comparison proof work.

Unfortunately, checking this regularity becomes far trickier on the boundary, for two main reasons:
first, even if the boundary is smooth, there is no available, natural inequality on the boundary
which can easily provide the needed regularity for subsolutions; the second one comes from the fact
that the stratified approach can handle very general domains, with non-smooth boundary, and even in
rather singular situations.

Let us now describe with more details these difficulties. 

\noindent\textsc{1. The lack of boundary inequalities ---} 
For strong stratified subsolutions, the regularity property in the $\R^N \times
(0,\Tf)$-case is a consequence of the $\F_*\leq 0$ inequality---or maybe of a similar inequality with
a suitable Hamiltonian in the cases of weak ones---and the ``normal controllability'' assumption. We
recall that the subsolution inequalities can be interpreted by ``all trajectories are sub-optimal''
or equivalently ``all choices of the dynamic are sub-optimal'' from the control point-of-view.  But
on the boundary, even if we only consider smooth boundaries at this point,
we cannot use all the trajectories, only those which stay in $\Omegb\times
(0,\Tf)$ \footnote{We refer the reader to Section~\ref{sect:natural.Ishii} for a discussion on the
difficulties connected to the $\F_*\leq 0$ inequality, even if it is in a slightly different
context.}.  Hence the $\F_*\leq 0$ inequality does not necessarily hold on $\domeg\times(0,\Tf)$,
and we have to find a suitable substitute providing the regularity of subsolutions.
Of course, the first natural idea is to use an Hamiltonian built out of all the dynamics pointing
inside $\Omega$ but this a priori optimal choice may be complicated in general since $\Omega$ is not
necessarily smooth so that the definition of ``pointing inside $\Omega$'' may be delicate.
Another choice is to use only ONE dynamic pointing inside $\Omega$ but we still need to make
precise the sense of this property. In the sequel, we mainly use this second option.

We point out that such difficulty with the regularity of sub and supersolutions on the boundary
already appears when studying state-constraints or Dirichlet boundary conditions, even in the most
standard  continuous cases. It is clear that such boundary conditions in the viscosity sense allow
the sub and supersolutions to have ``artificial values'' on $\domeg$ (they may be non-regular in the
language of this book). In particular, this is obviously the case for the subsolutions of ``classical''
state-constrained problems since they do not satisfy anything on the boundary. Therefore,  one way
or the other, some additional properties have to be imposed to solve this difficulty.

In the pionneering works of Soner \cite{Son1,Son2}, the ``cone condition'' appears involving both
some regularity of the boundary---an interior cone regularity---but also some property of the
dynamic---one of the control fields has to enter in this cone. From these first articles on the subject, it
was clear that a comparison result holds if the subsolution is not only $\Omega$-regular
at each point of the boundary but is also $K$-regular where $K$ is the interior cone.

Then, in their systematic study of Dirichlet problems, Perthame and the first author
\cite{BP1,BP2,BP3} obtain comparison results avoiding the direct use of a cone condition by showing,
under \NCe-type conditions, that on some parts of the boundary, these sub and supersolutions are
regular while on other parts, one can redefine their values on the boundary in order to transform
them into regular sub and supersolutions.

Perhaps closer to the spirit of what we suggest above, Ishii and Koike \cite{IK} have
formulated the state-constraints boundary condition in a different way, with an unusual subsolution
condition on the boundary, by looking only at dynamics which are pointing inside the domain on the
boundary: as can be guessed, their boundary condition ``$u_t + H_{in} \leq 0$'' avoids
non-regular subsolutions provided there is an inner dynamic and Lemma~\ref{RSub-1} below justifies this
natural idea.

Finally we point out that some results for first but also second-order equations are obtained by
Katsoulakis \cite{Kat} or Rouy and the first author \cite{GBER}: in \cite{GBER}, a blow-up argument
allows to show that the cone condition holds under suitable assumptions for first-order equations
and that a related property for the second-order case also holds.

\medskip

\noindent\textsc{2. Problems related to the geometry of the boundary ---}
This second difficulty comes from the wide variety of ``stratified domains'' we
can handle with the stratified approach. We refer the reader to Definition~\ref{def:stratdom} below
for a precise definition but let us already give several examples which show the particularities of the
stratified approach, its generality in terms of situations which can be taken into account and the
related difficulties.

Our first example, which is important since ambiguous, is given by
$$\Omega:=(-1,1)\times (-1,1) \setminus [0,1)\times \{0\}\;,$$
which is clearly not a smooth domain, see Figure~\ref{fig:pacman} below. 
The difficulty with the regularity of subsolutions appears at
the points of $(0,1)\times \{0\}$, where one has to carefully apply Definition~\ref{def:regular}. For
$r>0$ is small enough, if $x\in(0,1)\times\{0\}$, then $\Omega\cap B(x,r)$ has two
connected components and the regularity means a ``two-sided regularity'', like on $\Man{N}$
above in the case of tangentially (or locally) flattenable stratifications in $\R^N\times (0,\Tf)$. For $x=(0,0)$, on the contrary,
there is only one connected component and ``regular'' takes a more standard sense.

\begin{figure}[!htp]
    \begin{center}
    \includegraphics[width=0.3\textwidth]{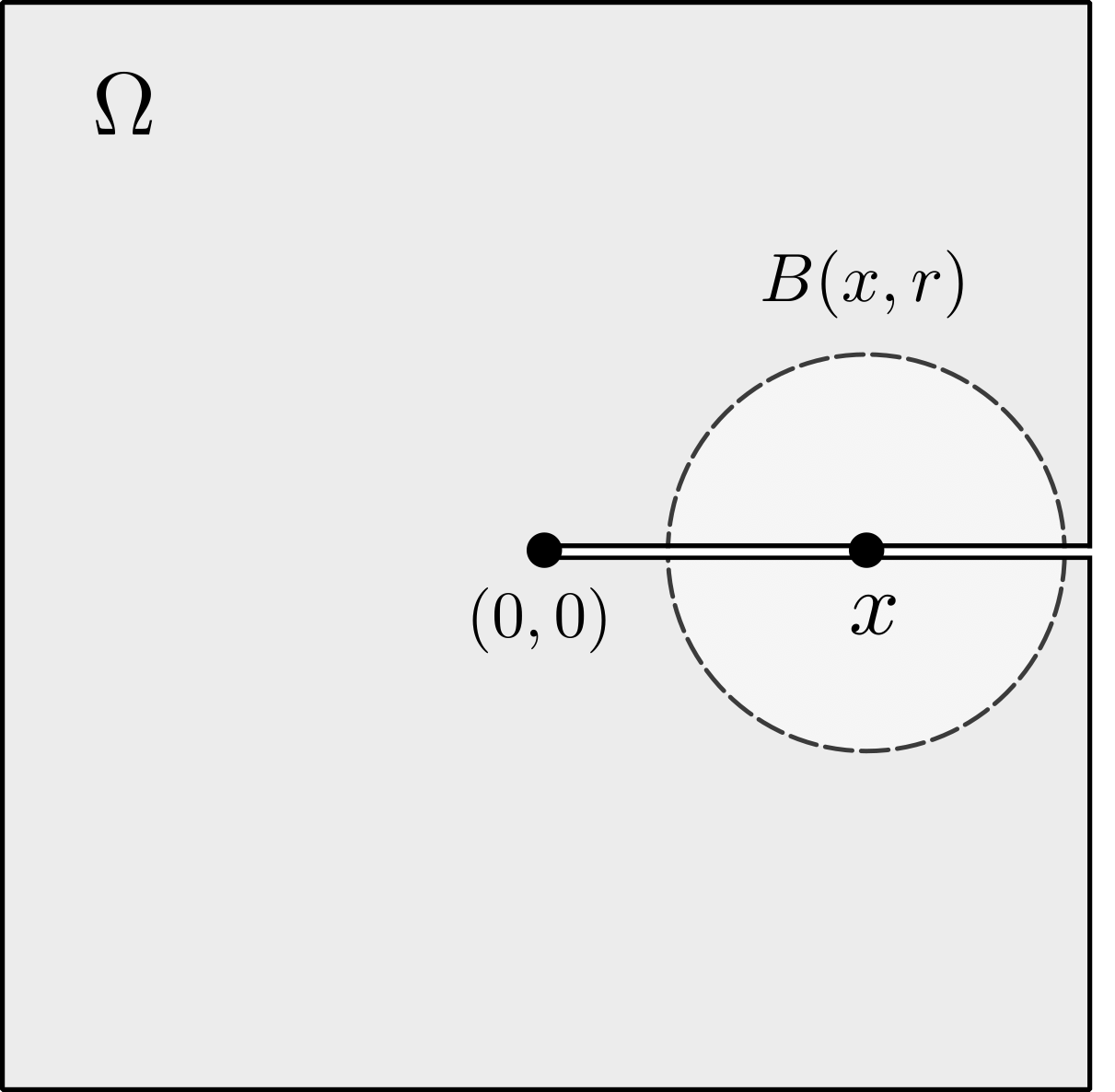}
    \caption{A non-smooth domain with a peculiar boundary.}
    \label{fig:pacman}    
    \end{center}
\end{figure}

But besides of this remark, a more intriguing question could be: are we right to consider
$[0,1)\times \{0\}$ as a part of the boundary? The answer is (more or less) that it does not matter!
Indeed, everything here is a matter of interpretation: either we can keep the idea that it is part of
the boundary; or we may consider that it is a part of an inside stratification. With the stratified
approach, there is no real difference between the ``equation inside the domain'' and the ``boundary
condition'', and therefore all the interpretations lead to the same formulation.

Next, the notion of stratified domain allows to treat---under suitable assumptions---non connected
domains which are connected through their boundaries. The simplest example being
$$ \Omega=\left[(-1,0)\times (-1,0)\right] \cup \left[(0,1)\times (0,1)\right]\;,$$
see Figure~\ref{fig:daisy} below (on the left).

\begin{figure}[!htp]
    \begin{center}
    \includegraphics[width=0.5\textwidth]{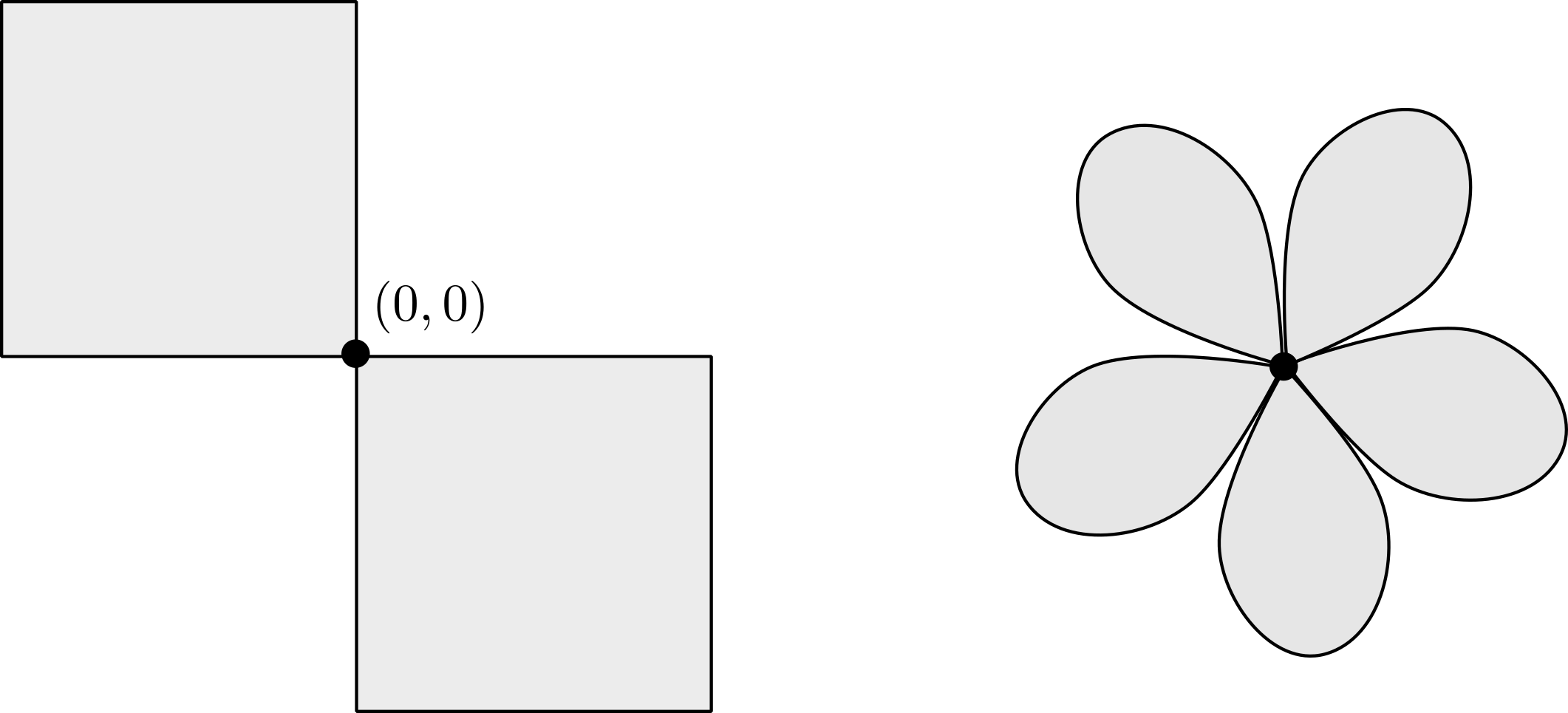}
    \caption{Daisy-like configurations.}
    \label{fig:daisy}    
    \end{center}
\end{figure}

Here $(0,0) \in \domeg$  belongs to the closure of both connected components of $\Omega$ and again
one has to apply carefully Definition~\ref{def:regular} to define regular subsolutions. Of course,
this example can be generalized as a daisy with several petals (the connected components of
$\Omega$), the center of the flower being the point $(0,0)$. At $(0,0)$ the regularity of
subsolutions has to be established \wrt each petal---\cf Figure~\ref{fig:daisy}, right.

\medskip

In all this part, we are going to avoid the difficulty connected to the regularity of sub and
supersolutions by proving several results only for regular ones, \ie for sub and/or supersolutions
whose boundary values are essentially limits of their values inside $\Omega$. Of course, the next
natural question is to identify some stable viscosity inequalities implying that, in particular,
subsolutions are ``regular'': we refer the reader to Section~\ref{abl} for a discussion. For
supersolutions, this regularity is treated in a more classical way.

\section{A second difficulty: initial and boundary data interaction}
\label{sect:initial.boundary.interaction}

In the $\R^N$-case, defining the initial data does not create any major difficulty: computing the
solution at time $t=0$ or comparing a sub and a supersolution just consist in studying a stationary
stratified problem with similar methods as for $t>0$, and with similar assumptions, which does not
lead to impose unnatural conditions. It may be thought that the same is true for state-constrained
problems, where the additional difficulties on the boundary of the domain are the same as for $t>0$.

Unfortunately, there are specific and unavoidable issues with $\domeg \times \{0\}$, which do
not depend on the approach we use. There are indeed well-known difficulties in initial-boundary
value problems: the compatibility of initial and boundary conditions in the case of Dirichlet
problems, the interaction of the initial condition and the Neumann boundary condition and, in
control problems with discontinuities, the possibility of having a specific control problem on
$\domeg \times \{0\}$, allowed by the upper-semicontinuity of $\BCL$.

In addition, in our general approach $\domeg \times \{0\}$ is itself a stratified set so the normal
controllability assumption should hold in a neighborhood of $\domeg \times \{0\}$, preventing an
Hamiltonian $\F_{init}(x,r,p_x)$ of the form $r-u_0(x)$ in $\Omega \times \{0\}$ to be admissible.
Nevertheless, this is a natural situation which should be handled by an appropriate treatment.

For this reason, we perform a specific study of the problem at time $t=0$, leading to restrict
ourselves to the two following cases.
\begin{enumerate}
    \item[(A)] The analogue to the $\R^N$-case, where the ``good assumptions'' are
        satisfied up to time $t=0$, in particular the normal controllability ones. In this case, we have
        to solve a stationary state-constrained problem to compute the initial data at time $t=0$
        and, of course, the regularity of subsolutions is a problem on $\domeg \times \{0\}$ as
        it is on $\domeg \times(0,\Tf)$.
        
     \item[(B)] The case of a Cauchy problem where the initial condition is---or can be reduced
        to---$u(x,0)=u_0(x)$ on $\Omegb \times \{0\}$ where $u_0 \in C(\Omegb)$ and for which we know
        that any subsolution $u$ and any supersolution $v$ satisfy $u(x,0)\leq u_0(x)\leq v(x,0)$ on
        $\Omegb \times \{0\}$.

\end{enumerate}

In order to study these two cases, we use the two following assumptions: we still use the notation
\HBASF when this assumption is satisfied with $\R^N$ replaced by $\Omegb$, while
\HBASFstar means that only Hypotheses \HBASF-$(i)$-$(ii)$ hold, again with $\R^N$ replaced by $\Omegb$. 

The key difference is that \HBASF contains an assumption for $t=0$ while it is not the case for
\HBASFstar, but we will complement it with the specfic Cauchy Problem initial condition
when we use it---see \HBAIDCP below.

%
%

\chapter{Stratified Solutions for State-Constrained Problems}
\label{chap-StratSC}

\abstract{This chapter contains every result on the stratified approach for state-constrained
problems: comparison result; discussions on the regularity of subsolutions at the boundary in order
to check the assumptions of the comparison result; applications to optimal control problems.}

\section{Admissible stratifications for state-constrained problems}

\index{Stratification!for state-constrained problems}

In this section, we extend the notions of {\em admissible stratification} for a Bellman Equation set
on $\Omegb \times (0,\Tf)$. Of course, the initial stratification $\Omegb\times\{t=0\}$ has to be treated
similarly and independently, as we remarked on page~\pageref{rem:MM0}. So, in the following we introduce
$$\M=(\Man{k})_{k=0..(N+1)}\quad\text{and}\quad \M_0=(\Man{k}_0)_{k=0..N}$$
where, for each $k=0..(N+1)$, $\Man{k}$ is a $k$-dimensional submanifold of $\R^N\times(0,\Tf)$ and
similarly, for each $k=0..N$, $\Man{k}_0$ is a $k$-dimensional submanifold of $\R^N$.

\begin{definition}\label{def:stratdom}\emph{--- $(\M,\M_0)$-stratified domains.}\smsp
    Let $\Omega$ be an open subset of $\R^N$ such that $\domeg=\partial [\Omegb^c]$.
    \begin{enumerate}
        \item[$(i)$] We say that $\Omegb \times (0,\Tf)$ is a $\M$-stratified domain if 
    $$\Omegb\times (0,\Tf) =\Man{0}\cup\Man{1}\cup\cdots\cup\Man{N+1}\; ,$$
    $$\domeg\times (0,\Tf) \subset \Man{0}\cup\Man{1}\cup\cdots\cup\Man{N}\; ,$$
    and the family $\tM=(\tMan{k})_k$ defined by $\tMan{k}=\Man{k}$ for $0\leq k\leq N$ and 
    $$\tMan{N+1}= \Man{N+1} \cup \left[\Omegb^c \times (0,\Tf)\right]$$ 
    is an \TFS of $\R^N\times (0,\Tf)$.
        \item[$(ii)$] We say that $\Omegb\times[0,\Tf)$ is a $(\M,\M_0)$-stratified domain if in
            addition to $(i)$, similar properties hold for $\Omegb$ and $\partial\Omega$ with respect to $\M_0$.
        \item[$(iii)$] When no confusion arises, we will just say that $\Omegb\times[0,\Tf)$ or
    $\Omegb\times(0,\Tf)$ are stratified domains, meaning that there are underlying stratifications
    $\M$ and $\M_0$ as above.
    \end{enumerate}
\end{definition}

In this definition, where we begin with an assumption on $\Omega$ whose aim is to avoid too
pathological cases, the only difference comes from the boundaries $\domeg \times (0,\Tf)$ and
$\partial\Omega\times\{0\}$. The following result explains the specific structure in the
state-constraints case\footnote{Of course, here, the value $k=(N+1)$ is excluded since no component of
$\Man{N+1}$ can be included in the boundary for obvious dimension considerations.}: 

\begin{proposition}\emph{--- Structure of the stratification.}\label{struct-mko}\smsp
    Let $\bar\Omega\times[0,\Tf)$ be a stratified domain.
    For any $0\leq k \leq N$, if $\Man{k}_i$ is a connected component of $\Man{k}$, then \\
    $$\text{either }\Man{k}_i\subset \domeg\times (0,\Tf)\text{ or }\Man{k}_i\subset
    \Omega\times (0,\Tf)\;.$$
    Of course, a similar property holds for $\Man{k}_0$, $0\leq k\leq(N-1)$.
\end{proposition}

\begin{proof} We do the proof only for the case $t>0$, the adaptations for $t=0$ being obvious, 
    so let us fix $0\leq k\leq N$.

    \smallskip

    \noindent\textbf{(a)} We first claim that 
    \begin{equation}\label{Struct-M-bord}
        (x,t) \in \Man{k} \cap \domeg\times (0,\Tf)\Rightarrow \exists r>0,\
        \Man{k} \cap B((x,t),r) \subset \domeg\times (0,\Tf)\;.
    \end{equation}
    This result being local, we can assume without loss of generality that there exists $r >0$ such that
    $\Man{k}\cap B((x,t),r) =[(x,t) + V_k]\cap B((x,t),r) $ where
    $V_k$ is a $k$-dimensional vector space. Then, \eqref{Struct-M-bord} is a consequence of
    the properties of an \TFS using the $\tM$-stratification:

    If, for some $v \in V_k$, $(x,t)+v \in [\Omega\times (0,\Tf)] \cap
    B((x,t),r)$, then there exists $0<\delta \ll r$ such that $B((x,t)+v,\delta) \subset
    [\Omega\times (0,\Tf)] \cap B((x,t),r)$. On the other hand, $B((x,t),\delta) \cap
    [\Omegb^c\times (0,\Tf)] \neq \emptyset$ and if $(x_\delta,t_\delta) \in B((x,t),\delta) \cap
    [\Omegb^c\times (0,\Tf)] \subset B((x,t),r) \cap \tMan{N+1}$, necessarily $(x_\delta,t_\delta) \in
    \tMan{N+1}$.

    By the properties of an \TFS, $(x_\delta,t_\delta) + V_k \subset \tMan{N+1}$ but
    $(x_\delta,t_\delta) \in \Omega^c\times (0,\Tf)$ and $(x_\delta,t_\delta) + v
    \in \Omega\times (0,\Tf)$ since $(x_\delta,t_\delta) + v \in B((x,t)+v,\delta)
    \subset [\Omega\times (0,\Tf)]$. Therefore $(x_\delta,t_\delta) + V_k$ has a point in
    $\domeg\times (0,\Tf)$ which is a contradiction since there is no point of $\tMan{N+1}$ on
    $\domeg\times (0,\Tf)$. 

    \smallskip

    \noindent\textbf{(b)} Now we come back to the Proposition. If $\Man{k}_i$ is a
    connected component of $\Man{k}$, there are two cases:\\
    -- either $\Man{k}_i\subset \Omega\times (0,\Tf) $ and we are done;\\
    -- or there exists $(x,t) \in \Man{k}_i\cap \domeg\times (0,\Tf)$.
    Now, if $\Man{k}_i$ is not entirely contained in $\domeg\times (0,\Tf)$, then the two subsets of 
    $\Man{k}_i$ defined by 
    $$\Man{k}_{i,1}=\Man{k}_i\cap \domeg\times (0,\Tf)\; , \,
    \Man{k}_{i,2}=\Man{k}_i\cap \Omega\times (0,\Tf)\; ,$$
    are both non-empty, open (by the above claim for $\Man{k}_{i,1}$) and we get 
    $\Man{k}_{i}=\Man{k}_{i,1}\cup \Man{k}_{i,2}$. A situation which is a contradiction with the
    connectedness of $\Man{k}_{i}$. Hence $\Man{k}_i \subset \domeg\times (0,\Tf)$ and the proof is complete.
\end{proof}

As a consequence, there is no interaction between $\Omega\times (0,\Tf) $ and $\domeg\times (0,\Tf)
$ through the stratification $\Man{k}$: no connected component can have some part intersecting
$\Omega\times (0,\Tf)$ and at the same time the complementary in $\domeg\times (0,\Tf)$. 

Notice though that the closure of some $\Man{k}_i \subset \Omega\times (0,\Tf)$ can
contain some points of $\domeg\times (0,\Tf)$. But then, they are contained in some $\Man{l}$ for some $l<k$.
As an example of such situation, where we drop the time-variable in order to simplify the example,
consider
$$ \Omega:= \{(x_1,x_2)\in \R^2;\ |x_1|+|x_2|<1\}\;.$$
We get a stratification of $\Omegb$ by setting $$\Man{0}=\{(0,-1), (1,0),(0,1),(-1,0), (0,0)\}\; ,$$
$$\Man{1}=\{(x_1,0);\  0<|x_1|<1\}\cup \{(0,x_2);\  0<|x_2|<1\}\cup \left(\domeg\setminus \Man{0}\right)\; ,$$
$$ \Man{2}=\Omegb\setminus \left(\Man{1}\cup \Man{0}\right)\; .$$
Of course, Proposition~\ref{struct-mko} applies but the two first connected components of $\Man{1}$
are not bounded away from $\domeg$.

As we will see later on, this will have a key importance in the definition of stratified
subsolutions: either we will consider interior points and, of course, this will be analogous to
the $\R^N \times (0,\Tf)$ case; or we will consider $\Fk$-inequalities at points of the boundary
which will not see any influence from $\Omega \times (0,\Tf)$. Indeed, in this last case $\Man{k}$
is included in $\domeg$ in a neighborhood of such points, so that these inequalities are just
``tangent'' inequalities.

\section{Stratified solutions and a basic comparison result}

In this section, we define the notion of stratified solution in the context of state-constrained
problems in full generality, even if we are going to use it only in $\bar\Omega\times(0,\Tf)$ in the
``basic'' comparison result we give at the end of the section.  As in the $\R^N$-case, we present
both a weak and a strong notion. We keep the same notations as in $\R^N$ for sub and supersolutions,
namely \SSup,\wSSub,\sSSub because the definition is actually the same: the case $\Omega=\R^N$ can
be viewed just as a particular case here. 

In the following, $\Omegb\times[0,\Tf)$ is a stratified domain associated to the 
collections of manifolds $\M$ for $t\in(0,\Tf)$, and $\M_0$ for $t=0$.

\index{Stratified solutions!definition (state-constaints case)}

\begin{definition}\emph{--- Stratified sub/supersolutions for state-constrained problems.}\smsp
    \noindent{\bf 1. ---} \SSup A locally bounded function $v : \Omegb \times [0,\Tf) \to \R$ is a
    stratified supersolution of \begin{equation}\label{BE-SC}
        \F(x,t,w,Dw)= 0\quad \hbox{on  }\Omegb \times [0,\Tf)\; ,
    \end{equation}
    if $v$---or equivalently $v_*$---is an Ishii supersolution of this equation on 
    $\Omegb \times [0,\Tf)$.\\[2mm]
    {\bf 2. ---} \wSSub A locally bounded function $u : \Omegb \times [0,\Tf)$ is a {\em weak}
    stratified subsolution of Equation~(\ref{BE-SC}) if
    \begin{enumerate}
        \item[$(a)$] for any $k=0,...,(N+1)$, $u^*$ is a viscosity subsolution of
        $$\F^k(x,t,u^*,D_xu^*) \leq 0 \quad \hbox{on  }\Man{k},$$
        \item[$(b)$] similarly for $t=0$ and $k=0..N$, $u^*(x,0)$ is a viscosity subsolution of 
            $$ \F_{init}^k(x,u^*(x,0),D_xu^*(x,0)) \leq 0 \quad \hbox{on  }\Man{k}_0 .$$
    \end{enumerate}
    \noindent{\bf 3. ---} \sSSub
    A locally bounded function $u$ is a {\em strong} stratified subsolution of Equation~(\ref{BE-SC})
    if it is a weak stratified subsolution which satisfies additionaly
    \begin{enumerate}
        \item[$(a)$] $\F_*(x,t,u^*,Du^*)\leq0$ in $\Omega\times(0,\Tf)$;
        \item[$(b)$] $(\F_{init})_*(x,u^*,D_xu^*)\leq 0$ in  $\Omega$\;.
    \end{enumerate}
    \noindent\textbf{4. ---} A weak or strong stratified solution is a function which is both a
    \SSup and either a \wSSub or a \sSSub.
\end{definition}

    In addition, we will say that $u$ is a strict (weak or strong) stratified subsolution 
    if the $\leq 0$-inequalities are replaced by a $\leq -\eta <0$-inequality 
    where $\eta>0$ is independent of $x$ and $t$.

Let us make several remarks on the definition.
\begin{enumerate}
    \item[$(i)$] The supersolution definition is just the classical
Ishii inequality, up to the boundary $\domeg \times (0,\Tf)$ as it is
classical for state-constrained problems. Of course at time $t=0$, the analogue
of Proposition~\ref{visc-ineq-init} implies that $\F$ can be replaced by
$\F_{init}$.

\item[$(ii)$] For the subsolution case, there is no change in $\Omega \times (0,\Tf)$, 
the main feature of stratified subsolutions are preserved, \ie we have to super-impose
$\F^k$-inequalities on all $\Man{k}$ (including at time $t=0$). What may be more suprising and
unusual in this state-constraints framework is the fact that there are subsolutions inequalities
on $\domeg \times (0,\Tf)$. But, on one hand these inequalities concern $\Man{k}\cap [\domeg
\times (0,\Tf)]$ for $k=0,...,N$ and therefore they take into account only the
dynamics which stay on $\Man{k}$, \ie on $\domeg \times (0,\Tf)$; on the other hand,  taking into
account these inequalities on the boundary is not a real difficulty here as long as we deal with
``regular subsolutions on the boundary''. This notion is defined precisely below, it is the natural
extension of the notion of regular subsolutions that we have seen in Part~\ref{stratRN}.

\item[$(iii)$] In fact, the new difficulty which is caused by the boundary is the following: as
we saw in the $\R^N$-case, the regularity of subsolutions is ensured for instance by the inequality
$\F_*\leq 0$ and \NCe, which are quite natural and allow to prove that \emph{weak} and \emph{strong}
subsolutions are the same.

We also mentioned that such regularity may come from other Hamiltonian inequalities or specific
properties depending on the situation.

And precisely here, on $\domeg \times (0,\Tf)$, the $\F_*\leq 0$ inequality cannot be expected in
general: since we only consider state-constrained trajectories, the outward pointing dynamics on
$\partial\Omega$ should be excluded from the computation of $\F_*$ there. 
This makes the regularity of subsolutions a real issue: we have to find a way to prove that
subsolutions are ``regular'' on the $\Man{k}$-components which lie on the boundary. As we
already mention it above, this is THE additional difficulty for state-constrained problems
in $\Omegb \times (0,\Tf)$.
\end{enumerate}

To address this question, we need to introduce a notion of boundary
regularity for subsolutions.

\begin{definition}\label{def:reg.boundary}
    Let $\Omegb\times[0,\Tf)$ be a stratified domain. We say that 

    \noindent $(i)$ an \usc function $u:\Omegb\times [0,\Tf)\to \R$
    is regular at the boundary $\domeg\times(0,\Tf)$ with respect to the stratification $\M$ if 
    $$\text{for any $1\leq k\leq N$,  $u$ is regular on }[\domeg\times(0,\Tf)] \cap \Man{k}\;.\ 
    \footnote{Here Definition~\ref{def:regular}-$(iii)$ is used with $A=\Omegb\times (0,\Tf)$
    and $\mathcal{E}=[\domeg\times(0,\Tf)] \cap \Man{k}$.}$$
    $(ii)$ an \usc function $u:\Omegb\times [0,\Tf)\to \R$
    is regular at the boundary $\domeg\times\{0\}$ with respect to the stratification $\M_0$ if 
    $$\text{for any $0\leq k\leq (N-1)$,  $u(x,0)$, is regular on }\domeg \cap \Man{k}_0\;.\ 
    \footnote{Here Definition~\ref{def:regular}-$(iii)$ is used with $A=\Omegb$ and 
    $\mathcal{E}=\domeg \cap \Man{k}_0$.}$$ 
\end{definition}
For the sake of simplicity, we will omit the mention ``with respect to $\M$/$\M_0$''.

We then conclude this part by a ``basic'' comparison result which has to be complemented by a
specific study of the problem at time $t=0$, \cf Section~\ref{sst0}. 
\index{Comparison result!for stratified solutions, state-constraints case}

\begin{theorem}\label{thm:basic.omega}\emph{--- Comparison in stratified domains.}\smsp
    Let $\Omegb\times[0,\Tf)$ be a stratified domain, assume that \HBASFstar holds and
    let $u$ be an \usc\ \wSSub, $v$ a \lsc\ \SSup such that
    \begin{equation}\label{cidSC}
    u(x,0)\leq v(x,0)\quad \hbox{on  }\Omegb.
    \end{equation}
    If $u$ is a regular subsolution in $\Omega \times (0,\Tf)$ and at the boundary
    $\domeg\times(0,\Tf)$, then $$ u(x,t)\leq v(x,t)\quad \hbox{on  }\Omegb\times[0,\Tf).$$
    In the case of strong subsolutions, the result holds for subsolutions which are regular at the boundary.
\end{theorem}

As we point out above, ``strong subsolutions'' are necessarily regular in $\Omega \times (0,\Tf)$.
Nevertheless, we have to keep the assumption that they are regular at the boundary since this is not
automatic.

As the reader may guess, the proof is almost exactly the same as the proof of
Theorem~\ref{comp-strat-RN} and it is easy to understand why: the fact that some
parts of the stratification are located on the boundary does not cause any problem and
the key ingredients were already used in the $\R^N$-case. The only difference comes
from the regularity of the subsolution at the boundary whose aim is, of course, to eliminate
``artificial values'' there.

As we will see below, this condition is analogous to the ``cone condition'' which is used in
state-constrained or Dirichlet problems for standard continuous equations. We will see in
Section~\ref{abl} how an analogue of the $\F_*$-inequality and Proposition~\ref{reg-sub} for
boundary points can be used to obtain some regularity property. This point may be important also for
stability reasons: while proving that the limit of a sequence of \wSSub is still a \wSSub may be
relatively easy, in order to use this convergence, the comparison result requires the limit \wSSub
to be regular. And this fact may be more complicated to prove, except if the conditions of
Lemma~\ref{RSub-1} are satisfied by the sequence of \wSSub in a suitable uniform way.

\section{On the boundary regularity of subsolutions}\label{abl}

\index{Regularity of subsolutions!on the boundary}
As we keep pointing out, the ``regularity'' of subsolutions plays a central role since this is a
keystone argument of the comparison result in the stratified case when we deal with subsolutions. We
recall that such regularity allows to obtain continuous subsolutions after ``tangential
regularization'' by sup-convolution.  In $\R^N$, this property is, in
general, a consequence of the standard Ishii subsolution inequality $\F_* \leq 0$, provided that the
normal controllability assumption is satisfied, \cf Proposition~\ref{reg-sub}. In other words,
strong stratified subsolutions are regular weak subsolutions, \cf Section~\ref{sec:regstratsub}. Of
course, in the present context, the same is true if we consider parts of the stratification which
are inside the domain $\Omega \times (0,\Tf)$.

However, the situation is a completely different on $\domeg \times (0,\Tf)$ since the subsolution
inequality $\F_* \leq 0$ does not hold, in general, on the boundary. It is replaced by the
$\F^k$-ones, involving only tangential dynamics which, therefore, cannot give information on the
values of the subsolution in $\Omega\times (0,\Tf)$ near the boundary. Actually, it is well-known
that, even in classical cases, (sub)solutions of the Dirichlet problems may have ``artificial''
values on the boundary which have no connections with interior ones---see \cite{BP1,BP2,BP3} or
\cite{Ba}. We come back on that classical case below.

We propose here two ways to get such connection between interior and boundary values: 
\begin{enumerate}
    \item[$(i)$] the first one is when some \adhoc inequalities on
        the boundary play the role of the ``$\F_* \leq 0$''-one, allowing to prove the regularity of
        subsolutions;
        \item[$(ii)$] the second one, inspired by the continuous case, is completely different: it consists
        in redefining the subsolution on the different portions $\Man{k}$ of the boundary in order to get
        the desired regularity property satisfied.
\end{enumerate}
Of course this second way is far more restrictive since it requires that no real discontinuity, in
terms of $\BCL$, is present on the boundary. But it may be useful though, since the stratified
approach allows non-smooth boundaries.

Finally, as we already mentioned it above, the question of the regularity of subsolutions on the
boundary can be even more delicate if some points of the boundary belong to the closure of several
connected components of $\Omega \times (0,\Tf)$. We will not try to answer this question in a
general way but we introduce below the notion of \emph{quasi-regular boundary} in order to be
able to fully apply strategy $(i)$ or $(ii)$ above.

\subsection{An inward-pointing cone condition}

In order to present our first result let us introduce some truncated cone.
The positive vectorial cone of height $\tau>0$ and aperture $\delta>0$ around direction $e\in\R^N$
is given by: 
$$\mathcal{C}^\tau(e,\delta):=\bigcup_{0< h\leq\tau}B(h e,h\delta)\;.$$
Of course, such vectorial cones can be transformed into affine cones by just adding $x\in\R^N$,
which amounts to translate $B(he,\delta h)$ into $B(x+he,h\delta)$. Notice that here we do not
require $e$ to be of unit length for simplicity. Figure~\ref{fig:cone.cond} below illustrates the cone
condition described in Lemma~\ref{RSub-1}, at time $t=t_0$.

\begin{figure}[!htp]
    \begin{center}
    \includegraphics[width=0.5\textwidth]{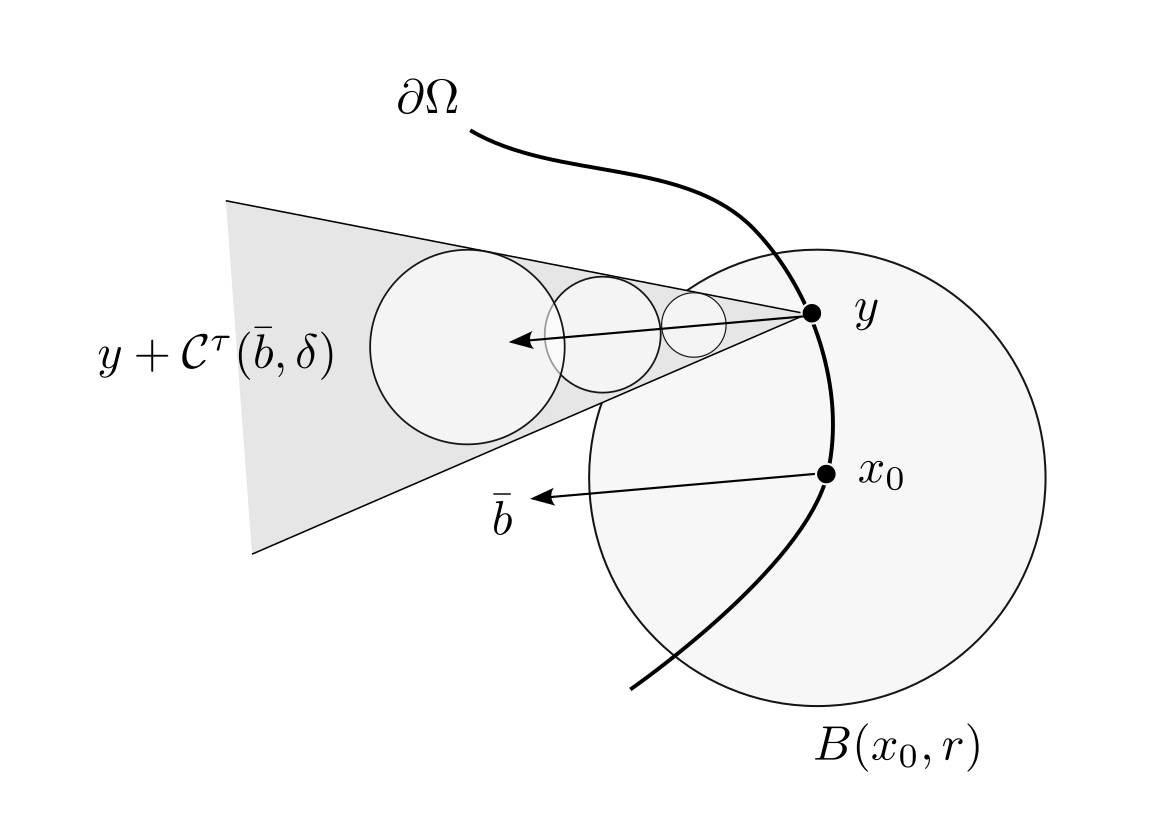}
    \caption{Inward pointing cone condition}
    \label{fig:cone.cond}    
    \end{center}
\end{figure}

\begin{lemma}\label{RSub-1} 
    Assume that $\overline \Omega\times (0,\Tf)$ is a stratified domain and that
    $(x_0,t_0) \in \Man{k} \cap [\domeg\times (0,\Tf)]$. Assume also that there
    exist $r,{\overline M},\tau,\delta >0$ and a continuous function 
    $b=(b^x,b^t):[\overline \Omega\times (0,\Tf)] \cap B((x_0,t_0),r) \to\R^{N+1}$
    such that, for any $y \in \Omega \cap B(x_0,r)$
    \begin{equation}\label{cone} 
        y+\mathcal{C}^\tau\big(\bar b^x,\delta\big)\subset \Omega\quad\text{where}\quad
        \bar b^x:=b^x(x_0,t_0)\; .
    \end{equation}
 If $u$ is a subsolution of
    \begin{equation}\label{b-eq}
-b(x,t) \cdot D u \leq {\overline M} \quad \hbox{on  } [\domeg \times (0,\Tf)]\cap
    B((x_0,t_0),r)\ \footnote{in the sense that, for any smooth test-function
    $\phi$ if $u-\phi$ has a local maximum point relatively to $\Omegb \times
    (0,\Tf)$ at $(x,t)$ then $-b(x,t) \cdot D\phi(x,t) \leq {\overline M}$.}\; ,
    \end{equation}
    then, for any $k\leq l <N+1$ and $(x,t) \in \Man{l}\cap [\domeg \times (0,\Tf)]\cap
    B((x_0,t_0),r)$, $u$ is $(\Man{l+1}\cup\cdots\cup \Man{N+1})$-regular at $(x,t)$. More precisely, we have
    \begin{equation}\label{reg-sub-Hk}
        u(x,t) =\limsup\Big\{u(y,s);\ (y,s)\to(x,t),\ (y,s) \in \Man{l+1}\cup \cdots\cup
        \Man{N+1}\Big\}\;.
    \end{equation}
\end{lemma} 

Notice that, referring to Definition~\ref{def:regular}-$(iii)$, Lemma~\ref{RSub-1} does not states
that $u$ is \emph{regular} on $\Man{k}$. Indeed, $\Man{k+1}\cup \cdots \cup \Man{N+1}$ may
have several connected components, so that \eqref{reg-sub-Hk} is not enough to ensure such
regularity. This will lead to the introduction of Assumption~\QRB later on. However, even if we are
in this ``bad'' situation with several connected components, it may be natural to try to apply the
same type of ideas with cones lying in each of the connected components of $\Omega$.

Before proving this lemma, we want to point out that, as the proof is going to show, this is a very
basic result; a more interesting point would be to give general and, if possible, natural conditions
under which a subsolution of the stratified problem is a viscosity subsolution of an equation like
\eqref{b-eq}. Of course, but this has to be formulated a little bit more precisely, such property is
in general a consequence of $(i)$ an interior cone condition like \eqref{cone} and $(ii)$ the normal
controllability assumption, together with a suitable compatibility between the two.

\begin{proof} 
    In order to simplify the presentation we are just going to prove \eqref{reg-sub-Hk} for $l=k$ and
    $(x,t)=(x_0,t_0)$, the proof for the other points being analogous.

    \smallskip

    \noindent\textbf{(a)} 
    We first claim that $\bar b=b(x_0,t_0)=(b^x(x_0,t_0),b^t(x_0,t_0))$ cannot be in
    $V_k=T_{(x_0,t_0)}\Man{k}$.

    This property is an easy consequence of Claim~\eqref{Struct-M-bord} in the proof of 
    Proposition~\ref{struct-mko}: indeed, otherwise we would have that, for small $\tau>0$, the
    distance from $(x_0, t_0)+\tau \bar b $ to $\Man{k}$ would be a $o(\tau)$ which would contradict 
    the cone assumption which implies that the distance of $(x_0, t_0)+\tau \bar b $
    to $\domeg \times (0,\Tf)$---and therefore to $\Man{k}$---is at least $\delta \tau$ 
    with $\delta>0$ because $x_0+\mathcal{C}^\tau\big(\bar b^x,\delta\big)\subset \Omega$.

    As a consequence, there exists a vector $e \in \R^{N+1}$, such that $e$ is orthogonal to
    $T_{(x_0,t_0)}\Man{k}$ and $\bar b\cdot e >0$. Then, in a small compact neighborhood of
    $(x_0,t_0)$, we consider the function 
    $$ (x,t) \mapsto u(x,t)-\frac{|x-x_0|^2}{\eps^2}+\frac{1}{\eps}e\cdot(x-x_0,t-t_0)
    -\frac{|t-t_0|^2}{\eps^2}\; .$$
    We notice that, by Cauchy-Schwarz inequality,
    $$ -\frac{|x-x_0|^2}{\eps^2}+\frac{1}{\eps}e\cdot(x-x_0,t-t_0)
    -\frac{|t-t_0|^2}{\eps^2}\leq -\frac{|x-x_0|^2}{2\eps^2}
    -\frac{|t-t_0|^2}{2\eps^2}+\frac{1}{2}|e|^2\; ,$$
    and therefore all the maximum points of this function satisfy at least that $\displaystyle
    \frac{|x-x_0|^2}{2\eps^2} +\frac{|t-t_0|^2}{2\eps^2}$ remains bounded when $\e \to 0$ and
    therefore these maximum points necessarily converge to $(x_0,t_0)$.

    \smallskip

    \noindent\textbf{(b)} For $0<\eps\ll 1$, if \eqref{reg-sub-Hk} does not hold and if $|e|$ small
    enough\footnote{smaller than the jump size of $u$ on the boundary}, then this function
    necessarily achieves its maximum on $\Man{k}$ at $(\xe,\te)$. Since $e$ is orthogonal to
    $T_{(x_0,t_0)}\Man{k}$, we have
    $$  e\cdot(\xe-x_0)=o\left(|\xe-x_0|\right)\; ,$$ and,
    as a consequence of the maximum point property we have
    $$ \begin{aligned} 
       u(x_0,t_0) &\leq u(\xe,\te)-\frac{|\xe-x_0|^2}{\eps^2}-
       \frac{|\te-t_0|^2}{\eps^2}+\frac{1}{\eps}e\cdot(\xe-x_0)\\
       &= u(\xe,\te)-\frac{|\xe-x_0|^2}{\eps^2}-\frac{|\te-t_0|^2}{\eps^2}+ \frac{1}{\eps}o\left(|\xe-x_0|\right)\; .
    \end{aligned}$$
    Refining the above Cauchy-Schwarz inequality, we can use (at least) the arguments of
    Lemma~\ref{lem:cv-pen} to prove that the penalisation terms $\displaystyle
    \frac{|\xe-x_0|^2}{\eps^2},\frac{|\te-t_0|^2}{\eps}$ tend to $0$ when $\eps \to 0$.  In
    particular, we have an other proof of the convergence of $(\xe,\te)$ to $(x_0,t_0)$ with a
    better estimate for the rate of convergence. 

    \smallskip

    \noindent\textbf{(c)}
    Writing the \eqref{b-eq} subsolution inequality yields
    $$ -b(\xe,\te)\cdot\left(\frac{2(\te-t_0)}{\eps^2} ,
    \frac{2(\xe-x_0)}{\eps^2}\right)+ \frac{1}{\e}e\cdot b(\xe,\te)\leq {\overline M}\; , $$
    which gives, thanks to the previous properties
    $$ \frac{o(1)}{\eps} + \frac{1}{\e}e\cdot b(\xe,\te)\cdot e \leq {\overline M}\; .$$
    But, by the continuity of $b$, $b (\xe,\te) \cdot e \to b(x_0,t_0)\cdot e>0$, and we get a
    contradiction in this above inequality for $\eps$ small enough.  
\end{proof}

The next result shows how \eqref{b-eq} can be obtained and the kind of
compatibility conditions which are needed to get it, combining the cone condition
\eqref{cone} and the dynamic in the control problem.

\begin{lemma}\label{lem:sufficient.cone}
    Let $\overline \Omega\times (0,\Tf)$ be a stratified domain and 
    $(x_0,t_0) \in \Man{k}\cap [\domeg\times (0,\Tf)]$. We make the following assumptions:
    \begin{enumerate}
        \item[$(i)$] \HBCL, \NCBCL, \TCBCL hold;
        \item[$(ii)$] there exist $r, \tau , \delta >0$ and $b:[\overline \Omega\times
    (0,\Tf)] \cap B((x_0,t_0),r)\to\R^{N+1}$ continuous such that for any $y \in \domeg \cap
    B(x_0,r)$, cone condition \eqref{cone} holds;
    \item[$(iii)$] for any $(x,t) \in [\overline \Omega \times (0,\Tf)]\cap B((x_0,t_0),r)$,
        $b(x,t) \in  \B(x,t)$.
    \end{enumerate}
    If $u$ is a subsolution of the stratified state constraint problem in $[\Omegb \times (0,\Tf)]
    \cap B((x_0,t_0),r)$ and if \eqref{reg-sub-Hk} holds for any $(x,t) \in [\domeg \times (0,\Tf)] \cap
    B((x_0,t_0),r)$, then $u$ is a subsolution of \eqref{b-eq} for eventually a smaller $r$ and
    for some large enough constant ${\overline M}$ depending on the
    $\L^\infty$-norm of $u$ and on the constant $M$ which appears in \HBCL.  
\end{lemma} 

This corollary means that, in some sense, property \eqref{reg-sub-Hk} is equivalent to a natural
``control'' inequality (as it is the case in $\Omega$) and that such inequality should be
automatically extended to the boundary if the boundary values are the limit of the interior ones.

\begin{proof} 
    Thanks to \HBCL, \NCBCL and \TCBCL, we can use the regularization procedure of
    Section~\ref{sect:sup.reg}, so that we can assume without loss of generality that $u$ is
    Lipschitz continuous on $[\Omegb \times (0,\Tf)] \cap B((x_0,t_0),r)$. We point out that
    \eqref{reg-sub-Hk} plays a key role in this property in order to avoid any discontinuity on the
    boundary.

    \smallskip

    \noindent\textbf{(a)}
    If $\phi$ is a smooth test-function and if $(\xb,\tb) \in [\domeg  \times (0,\Tf)] \cap
    B((x_0,t_0),r)$ is a strict local maximum point of $u-\phi$ in $[\Omegb  \times (0,\Tf)] \cap
    B((x_0,t_0),r)$, we consider the function 
    $$ \Psi(x,t,y,s) = u(x,t) - \phi(y,s)- \frac{|x-y-\eps
    \tilde b^x|^2}{\eps^2} -\frac{|t-s-\e\tilde b^t|^2}{\eps^2}\; ,$$
    where $\tilde b=(b^x(\xb,\tb),b^t(\xb,\tb))$. We notice that, taking possibly a smaller $r$ and changing $\delta$ in $\delta/2$,
    \eqref{cone} holds with $\bar b=b(x_0,t_0)$ replaced by $\tilde b$.
    
    The function $\Psi$ achieves its maximum at some point
    $(\xe,\te,\ye,\se)$. Since $u$ is Lipschitz continuous, $u(\xb +\eps \tilde b^x,\tb+\e\tilde b^t) =
    u(\xb ,\tb) + o_\eps(1)$ and therefore 
    $$ u(\xb,\tb) - \phi(\xb,\tb) + o_\eps(1) \leq \Psi(\xb +\eps \tilde b^x,\tb+\e \tilde b^t,\xb ,\tb) 
    \leq \Psi (\xe,\te,\ye,\se) \leq u(\xb,\tb) - \phi(\xb,\tb) +  o_\eps(1)\; ,$$
    the last inequality coming from the facts that $\xe-\ye,\te-\se$ are $O(\eps)$, $u$ is Lipschitz continuous and $(\xb,\tb)$
    is a maximum point of $u-\phi$.

    \smallskip

    \noindent\textbf{(b)}
    By Lemma~\ref{lem:cv-pen} (or at least by using the underlying arguments), we deduce that as
    $\eps\to0$, not only $(\xe,\te),(\ye,\se)\to (\xb,\tb)$ but also 
    $$ \frac{|\xe-\ye-\eps \tilde b^x|^2}{\eps^2} +\frac{|\te-\se-\e\tilde b^t|^2}{\eps^2} \to 0\;.$$
    In particular for $\e>0$ small enough, 
    $\xe \in B(\ye + \eps \tilde b,\delta\eps)\subset y_\e + \mathcal{C}^\tau(\tilde b,\delta) 
    \subset \Omega$ since $|\xe-\ye-\eps \tilde b^x|=o(\eps)$.
    We then write down the viscosity subsolution inequality for $u$
    $$  \F_*(\xe, \te, u(\xe,\te), (p_\eps,\alpha_\eps))\leq 0\; ,$$
    where $\alpha_\eps=2(\te-\se-\e\tilde b^t){\eps}^{-2}= \phi_t (\ye,\se)$
    by the maximum point property in $s$, while 
    $p_\eps = 2(\xe-\ye-\eps \tilde b^x)\eps^{-2} = 
    D \phi (\ye,\se)$ if $\ye \in \Omega$, but not necessarily if $\ye \in \domeg$---more about
    this case below. 

    \smallskip

    \noindent\textbf{(c)} In order to estimate $\F_*$, we recall that for any $(\hat b,\hat c,\hat l)\in\BCL(x,t)$,
    $$ \F (x,t,r,p) \geq -\hat b\cdot p + \hat cr - \hat l\;.$$
    In particular, by Assumption~$(iii)$, we see that
    \begin{equation}\label{prop-Fls-int}
        \F_* (x,t,r,p) \geq  - b(x,t) \cdot p - \tilde M
    \end{equation}
    for some constant $\tilde M$ since $c,l$ are bounded and since, for $r$, $u$ is a bounded
    subsolution. Therefore, 
    \begin{equation}\label{ineq-Fls-int}
        0\geq  \F_* (\xe,\te,u(\xe,\te),(p_\e,\alpha_\e)) \geq -  b(\xe,\te) \cdot (p_\eps,\alpha_\e)-
        \tilde M
    \end{equation}
    and we conclude easily if we know that $\ye \in \Omega$ at least for a
    subsequence of $\eps$ tending to $0$, with ${\overline M}=\tilde M/\kappa$.

    If $\ye \in \domeg$, we first notice that, for $0<\tau \ll 1$, $\ye +\tau b^x(\xe,\te)\in
    \Omega$ as a consequence of the cone condition and then we use  $$\Psi(\xe,\te,\ye +\tau
    b^x(\xe,\te),\se +\tau b^t(\xe,\te)) \leq \Psi (\xe,\te,\ye,\se)\; ,$$ and this variation gives
    $-b(\ye,\te)\cdot D \phi (\ye,\se) \leq - b(\ye,\te) \cdot (p_\eps,\alpha_\e)$. The conclusion
    follows easily from \eqref{ineq-Fls-int} and the proof is complete.
\end{proof}

\begin{remark}\label{rem:ineg-sub-cone-int} 
    The conclusion of Lemma~\ref{lem:sufficient.cone} is the boundary inequality \eqref{b-eq}; we
    state it in this way since this is the information which is useful for applying
    Lemma~\ref{RSub-1}.  But, in fact, the subsolution inequality \eqref{b-eq} holds not only on
    $\domeg  \times (0,\Tf)] \cap B((x_0,t_0),r)$ but also in $\Omegb  \times (0,\Tf)] \cap
    B((x_0,t_0),r)$; this is obvious from inequality~\eqref{prop-Fls-int}. Such property may be
    useful for recovering the regularity of a subsolution obtained through a stability result like
    the half-relaxed limits method. We refer the reader to Chapter~\ref{chap:networks} for a
    situation where this remark plays a key role.
\end{remark}

\subsection{Redefining the boundary values}

Now we turn to the second possibility which is more restrictive but which may be interesting in exit
time and associated Dirichlet problems.  In order to simplify the formulation of this result, we use
the notation $\QOm=B((x_0,t_0),r)\cap \left (\overline \Omega\times (0,\Tf)\right)$ for a
localization cylinder.
\begin{lemma}\label{lem:rbv}
    Assume that $\overline \Omega\times (0,\Tf)$ is a stratified domain. Let 
    $(x_0,t_0) \in \Man{k}$ and $r>0$ such that 
    $\QOm \subset \Man{k}\cup \Man{k+1}\cup \cdots \cup \Man{N+1}$.
    We make the following additional assumptions:
    \begin{enumerate}
        \item[$(i)$] There exists a set-valued map $\BCLp : \QOm \to \R^{N+3}$ such that
        \begin{enumerate}
            \item[$(a)$] for any $(x,t) \in \QOm$, $\BCLp(x,t) \subset \BCL(x,t)$.
        \item[$(b)$] for any $(x,t) \in \Man{k} \cap \QOm$ and $(y,s) \in \Man{l} \cap 
            \QOm$ for some $l\geq k$,
        if $(b,c,l)\in\BCLp(x,t)$ is such that
            $b\in T_{(x,t)}\Man{k}$, there exists $(b',c',l')\in\BCLp(y,s)$ with $b'\in T_{(y,s)}\Man{l}$
            satisfying
            $$  |b'-b| \leq O(|x-y| + |t-s|)\;,\, |c-c'|+|l-l'| =o(1)\;\text{as }
            (y,s)-(x,t)\to 0\;.$$
     \end{enumerate}
        \item[$(ii)$] For any $l=(k+1),..,(N+1)$, $u$ is an \usc subsolution of the $\F^{l}$-equation 
            in $\Man{l}\cap \QOm$.
    \end{enumerate}
    Then the function $\tilde u$ defined on $\Man{k} \cap \QOm$  by 
    $$
    \tilde u(x,t) := \limsup\Big\{u(y,s);\ (y,s)\to (x,t), \ (y,s) \in (\Man{k+1}\cup \cdots \cup
    \Man{N+1})\cap \QOm\Big\} \;
    $$
    satisfies 
    $$ \sup_{\substack{(b,c,l)\in\BCLp(x,t)\\ b\in T_{(x,t)}\Man{k}}}\big\{ -
    b\cdot D\tilde u +c\tilde u - l\big\} \leq 0\quad \hbox{on  }\Man{k} \cap \QOm\; ,$$
\end{lemma}

Let us now comment the admittedly strange formulation of Lemma~\ref{lem:rbv}, the idea being, one way
or the other, to obtain an $\F^k$-inequality on $\Man{k}$. One may have two cases in mind.

\noindent \emph{1. The case when $\BCLp=\BCL$.} 

With this choice, the result---which is indeed an
$\F^k$-inequality on $\Man{k}$---can be applied even to parts of $\Man{k}$ which lie inside $\Omega
\times (0,\Tf)$. But it is clear that the interest of Lemma~\ref{lem:rbv} is limited in this
situation since Assumption~$(i)$-$(b)$ means that there is no real discontinuity in $\BCL$ on
$\Man{k}$. Of course, a similar remark may be made for parts of $\Man{k}$ which lie on the boundary
$\domeg \times (0,\Tf)$ but here the existence of $\Man{k}$ is not necessarily connected to a
discontinuity in the equation (\ie in the $\BCL$): it may also come from a non-smooth boundary,
typically a corner. However, it is also clear in this situation that there is no particular
``boundary condition'' on $\Man{k} \cap \QOm$ since Assumption~$(i)$-$(b)$ implies that
the $(b,c,l)\in\BCL(x,t)$ such that $b\in T_{(x,t)}\Man{k}$ are just limits of the ones in
$\Man{k+1}\cup \cdots \cup \Man{N+1}$, and in particular in $\Man{N+1}$. Therefore this is a
restrictive case which typically applies to a non-discontinuous state-constraints case.

\smallskip

\noindent 2. \emph{The case of Dirichlet boundary conditions.}

Here the specific boundary condition is encoded as
$$ \BCL(x,t)=\overline{\mathrm{co}}\big(\BCLp(x,t) \cup \{(0,1,\varphi(x))\})
\quad \hbox{on  }\Man{k} \cap \QOm\;,$$
where $\varphi$ is a continuous function on $B((x_0,t_0),r) \cap \left(\domeg \times
(0,\Tf)\right)$. In general, the normal controllability assumption implies that $u(x,t) \leq
\varphi(x,t)$ on $B((x_0,t_0),r) \cap \left(\domeg \times (0,\Tf)\right)$ and obviously the same
inequality is satisfied by $\tilde u$ since $\tilde u \leq u$. Finally, the expected $\F^k$-inequality 
for $\tilde u$ on $\Man{k} \cap \QOm$ is obtained by combining the result of
Lemma~\ref{lem:rbv} with the fact that $\tilde u(x,t) \leq \varphi(x,t)$ on $B((x_0,t_0),r) \cap
\left(\domeg \times (0,\Tf)\right)$.

Such result, which is the analogue of a procedure used in Barles and Perthame \cite{BP2,BP3} (see
also \cite{Ba}) to ``clean'' the boundary values of the sub or supersolution, is particularly useful
for non-smooth boundaries. In practical use, one has to proceed via a reverse induction, first
redefining the subsolution on $\Man{N}\cap \left (\overline \Omega\times (0,\Tf)\right) $ and then
on $\Man{N-1}\cap \left (\overline \Omega\times (0,\Tf)\right) $...etc.

\begin{proof} 
    Let $\phi$ be a smooth test-function and $(x,t) \in \Man{k} \cap \QOm$, a strict local
    maximum point of $\tilde u-\phi$ on $\Man{k} \cap \QOm$. We want to prove that, for
    every $(b,c,l)\in\BCL'(x,t)$ with $b\in T_{(x,t)}\Man{k}$, 
    $$ - b\cdot D\phi(x,t) +c \tilde u(x,t) - l \leq 0\; .$$
    From now on, we fix such a $ (b,c,l)$.

    \smallskip

    \noindent\textbf{(a)}
    By definition of $\tilde u$, there exists a sequence $((x_\eta,t_\eta))_\eta$ in
    $(\Man{k+1}\cup \cdots \cup \Man{N+1})\cap \QOm$ such that $(x_\eta,t_\eta) \to (x,t)$
    and $u(x_\eta,t_\eta) \to \tilde u(x,t)$. As a consequence, for any $\e>0$ small enough, there
    exists $\eta$ such that $d((x_\eta,t_\eta),\Man{k})\leq \e^2$ and we set $\alpha :=
    [d((x_\eta,t_\eta),\Man{k})]^2$. Notice that $\alpha \to 0$ when $\e \to 0$.

    Next, we introduce the function
    $$(y,s)\mapsto u(y,s)-\phi(y,s)-\frac{\alpha}{d((y,s),\Man{k})}-\frac{d((y,s),\Man{k})}{\eps}\;.$$
    An easy use of Lemma~\ref{lem:cv-pen}  implies that, for $\e$ small enough, this function
    achieves its maximum at $(\xb,\tb)\in (\Man{k+1}\cup \cdots \cup \Man{N+1})\cap \QOm$, where 
    we drop the dependence of  $(\xb,\tb)$ in $\e$ for the sake of simplicity of notations.
    Moreover, 
    \begin{equation}\label{est.pen}
        u(\xb,\tb)\to \tilde u(x,t)\;,\quad  \frac{\alpha}{d((\xb,\tb),\Man{k})} +
        \frac{d((\xb,\tb),\Man{k})}{\eps}\to 0\;.
    \end{equation}

    \noindent\textbf{(b)}
    For $\e$ small enough, there exists a unique point $(\yb, \ovs) \in \Man{k} \cap
    \QOm$ such that $d((\xb,\tb),\Man{k})=\vert (\yb, \ovs)-(\xb,\tb)\vert$. Using
    Assumptions $(i)$-$(b)$, there exists $(b',c',l')\in\BCLp(\yb,\ovs)$
    with $b'\in T_{(\yb,\ovs)}\Man{k}$ satisfying
    $$  |b'-b| \leq O(|x-\yb| + |t-\ovs|)\;,\, |c-c'|+|l-l'| =o(1)\;\text{as }\e\to 0\;.$$
    For the same reason, if $(\xb,\tb)\in \Man{l}$, there exists $(b'',c'',l'')\in\BCLp(\xb,\tb)$ 
    with $b''\in T_{(\xb,\tb)}\Man{l}$ satisfying
    \begin{equation}\label{est.b''}
        |b''-b'| \leq O(d((\xb,\tb),\Man{k}))\;,\, |c'-c''|+|l'-l''| =o(1)\;\text{as }\e\to 0\;.
    \end{equation}
    With these notations, the $\F^l$-inequality at $(\xb,\tb)$ gives in particular
    $$-b''\cdot \left(D\phi(\xb,\tb) - \frac{\alpha Dd((\xb,\tb),\Man{k})}{[d((\xb,\tb),\Man{k})]^2} - 
    \frac{Dd((\xb,\tb),\Man{k})}{\eps}\right)+c'' u(\xb,\tb)-l''\leq 0\;.$$
    Combining \eqref{est.pen} and \eqref{est.b''} we are led to
    $$ -b'\cdot \left(D\phi(\xb,\tb) - \frac{\alpha Dd((\xb,\tb),\Man{k})}{[d((\xb,\tb),\Man{k})]^2} -
    \frac{Dd((\xb,\tb),\Man{k})}{\eps}\right) +c' u(\xb,\tb)-l' \leq o_\e(1)\;.$$

    \noindent\textbf{(c)}
    We notice that $Dd((\xb,\tb),\Man{k})\cdot b' =0$ since $b'\in T_{(\yb,\ovs)}\Man{k}$ and
    the gradient of the distance is orthogonal to $T_{(\yb,\ovs)}\Man{k}$---the reader can be
    even more convinced by this fact assuming that $\Man{k}$ is flat. This implies that
    $$ -b'\cdot D\phi(\xb,\tb) + c'u(\xb,\tb)-l' \leq o_\e (1)\;,$$
    and we conclude, by letting $\e$ tend to $0$, using that $(b',c',l') \to (b,c,l)$,
    $(\xb,\tb)\to (x,t)$ and $u(\xb,\tb)\to \tilde u(x,t)$
\end{proof}

\subsection{Quasi-regular boundaries}

We conclude Section~\ref{abl} with giving sufficient regularity conditions in the case where each
point of $\domeg \times (0,\Tf)$ belongs to the closure of only one connected component of $\Omega
\times (0,\Tf)$. More precisely, we use the hypothesis

\begin{assumption}{\QRB}{Quasi-regular boundary assumption.} \label{page:QRB}
For any $(x,t) \in \domeg \times (0,\Tf)$, if $(x,t) \in \Man{k}$, then there exists $r_0>0$ such that
$$ \left[ \Man{k+1}\cup \cdots \Man{N+1}\right] \cap B((x,t),r_0) \; \hbox{is connected}\; .$$
    \vspace*{-1cm}
\end{assumption}

With this assumption, the regularity of an \usc function $u:\Omegb \times (0,\Tf) \to \R$ on
$\Man{k}\cap [\domeg \times (0,\Tf)]$ just reduces to \eqref{reg-sub-Hk} and the previous
subsections provide sufficient conditions to get it. The result is the

\begin{corollary}\emph{--- Boudary regularity of subsolutions.}\smsp
    Let $\Omegb\times(0,\Tf)$ be a stratified domain, assume that \QRB and \HBASFstar hold and let
    $u$ be an \usc~\wSSub.
    \begin{enumerate}
        \item[$(i)$] If the hypotheses of Lemma~\ref{RSub-1}  hold on each point
            $(x,t)\in\partial\Omega\times(0,\Tf)$, then $u$ is a regular \wSSub.
        \item[$(ii)$] If the hypotheses of Lemma~\ref{lem:rbv} hold on each point
            $(x,t)\in\partial\Omega\times(0,\Tf)$, then $u$ can be redefined on
            $\partial\Omega\times(0,\Tf)$ so that it becomes a regular \wSSub. 
    \end{enumerate}
\end{corollary}

Notice that of course, there are situations where \QRB does not hold for which the above may apply.
In particular, if $(i)$ is satisfied in each connected component touching the boundary, then
regularity follows.

\section{Refined versions of the comparison result}
\label{sst0}

We start by the easiest case which is the analogue of the $\R^N$-one by assuming \HBASF, in other
words, that the ``good assumptions'' hold up to $t=0$.

\begin{theorem}
    Let $\Omegb\times[0,\Tf)$ be a stratified domain and assume that \HBASF holds. 
    Let $u$ be an \usc~\wSSub and $v$ be a \lsc~\SSup such that $u$ is a
    regular subsolution in $\Omega \times (0,\Tf)$, in $\Omega \times \{0\}$, at the boundary
    $\domeg\times(0,\Tf)$ and on $\domeg\times\{0\}$. Then, 
    $$u(x,0)\leq v(x,0)\quad\text{in}\quad\Omegb$$
    and therefore
    $$ u(x,t)\leq v(x,t)\quad \text{in}\quad\Omegb\times[0,\Tf).$$
    In the case of strong subsolutions, the result holds for subsolutions which are regular at the
    boundaries $\domeg\times(0,\Tf)$ and $\domeg\times\{0\}$.
\end{theorem}

The proof of this result is simple since the proof of \eqref{cidSC} follows from similar arguments
as for the proof of the ``basic'' comparison result and then it suffices to apply this ``basic''
comparison result.

The defect of the above result is that it is not of an easy use, even in a simple case like when
$\F_{init}(x,r,p_x)=r-u_0(x)$ in $\Omega \times \{0\}$, where $u_0\in C(\Omegb)$. Indeed, such
$\F_{init}$ does not satisfy \HBASF-$(iii)$, and more precisely \NCBCL in a neighborhood of $\domeg
\times \{0\}$. And this is not just a technical difficulty since a specific control problem may
exist on the boundary $\domeg\times\{0\}$---or even just the trace of a Dirichlet boundary condition
on $\domeg\times[0,\Tf)$---and be the source of a discontinuity for the value function. In these
cases, we cannot expect a comparison result to hold.

Hence the aim is to investigate cases where, for any sub and supersolution the following holds:
$$ u(x,0)\leq v(x,0)\quad \text{in}\quad\Omegb.$$
To do so, we first provide a pseudo-analogue of Proposition~\ref{visc-ineq-init} for points on
$\domeg\times\{0\}$, for which we use the following assumption:

\begin{assumption}{\HBAIDCP}{Basic Assumption on the Initial Data for the Cauchy Problem.}\label{page:HBAIDCP}
    There exists $u_0 \in C(\Omegb)$ such that
    $$\begin{cases} \F_{init}(x,r,p)=r-u_0(x) & \text{in } \Omegb \times \{0\}\;,\\
      \F^k_{init}(x,r,p)=r-u_0(x) & \text{on }\domeg \times \{0\}\;,\quad k=0..N\;.
    \end{cases}$$
\end{assumption}

The result is the

\begin{proposition}\label{visc-ineq-init-bord}
    Let $\Omegb\times[0,\Tf)$ be a stratified domain.
    \begin{enumerate}
            \item[$(i)$] Under assumption \HBASF, if $v: \Omegb \times [0,\Tf) \to \R$ is a \lsc~\SSup 
            then $v(x,0)$ is a supersolution of 
            $$\F_{init}(x,v(x,0),D_x v(x,0) )\geq 0 \quad\hbox{in }\Omegb\;.$$
            \item[$(ii)$] Under assumption \HBASFstar and \HBAIDCP, if $u: \R^N \times [0,\Tf] \to \R$
                is an \usc~\wSSub, then $$u(x,0)\leq u_0(x)\quad\hbox{in }\Omegb\;.$$
    \end{enumerate}
\end{proposition}

\begin{proof}
    We just sketch the proof since it is an easy adaptation of standard arguments, and in particular
    those of the proof of Proposition~\ref{visc-ineq-init} and, actually, the proof of $(i)$ follows
    from readily the same arguments.

    For $(ii)$, we have to transform $\F_{init}$ or $\F_{init}^k$-viscosity inequalities into usual
    inequalities. If $x\in \Man{k}_0$, we have just to consider the function 
    $$(y,s) \mapsto u(y,s)-\frac{|y-x|^2}{\e}\quad\hbox{on  }\Man{k}_0\; ;$$
    we have a sequence of local maximas $(\ye,\se)$ such that $(\ye,\se)\to (x,0)$ and
    $u(\ye,\se)\to u(x,0)$ as $\e\to 0$ and, thanks to Assumption \HBAIDCP, we have $u(\ye,\se) \leq
    u_0(\ye)$. The conclusion follows by letting $\e\to 0$.
\end{proof}

\begin{remark}\label{rem:pbidbord}
    Let us examine Assumption \HBAIDCP in the case of a standard Dirichlet problem
    $$
    u_t + H(x,t,D_x u)=0 \quad \hbox{in $ \Omega \times (0,\Tf)$}\; ,$$
    $$
    u(x,0) =u_0 (x) \quad  \hbox{in $\Omega$}\; ,$$
    $$
    u(x,t) =\varphi (x,t) \quad  \hbox{on $\domeg$}\; ,$$
    where $\Omega$ is a domain in $\R^N$---we may even assume that $\Omega$ is a smooth domain---,
    $u_0,\varphi$ are continuous functions and $H$ is a continuous Hamiltonian coming from a control
    problem. Clearly the computation of $\F_{init}(x,r,p_x)$ gives $r-u_0(x)$ if $x\in \Omega$ but,
    on the boundary, an interaction occurs between the initial and boundary data, which yields
    $$\F_{init}(x,r,p_x)=\max\big(r-u_0(x), \varphi(x,0)\big)\;.$$ 
    Obviously, Assumption \HBAIDCP is satisfied provided $\varphi(x,0)\geq u_0(x)$.

    Of course, a comparison result implies that the solution is continuous and clearly, in the
    control case, this solution should be the value function. In the above case, where we have both
    an exit cost $\varphi$ and a terminal cost $u_0$, if the exit cost satisfies $\varphi(x_0, 0) <
    u_0(x_0)$, the same inequality remains valid on a neighborhood of $x_0$. In this neighborhood,
    the controller should try to exit the domain in order to pay the cheapest cost $\varphi$. This is
    possible because of \NCBCL but only for points $(x,t)$ for which $x$ is close enough to the
    boundary. Hence, we see a discontinuity at the points separating the region where exiting is
    possible and those for which this is not the case.

    Therefore, \HBAIDCP seems a rather natural assumption and we refer the reader to
    Chapter~\ref{chap:RefBF} for various examples of the checking of \HBAIDCP which will be even
    more convincing.
\end{remark}

We end up with the
\begin{corollary}\label{cor:comp.BACPID}\emph{--- Refined version of the comparison result.}\smsp
    Let $\Omegb\times[0,\Tf)$ be a stratified domain and assume that assumptions \HBASFstar and
    \HBAIDCP hold. If $u$ is an \usc~\wSSub which is regular
    in $\Omega \times (0,\Tf)$ and at the boundary $\domeg\times(0,\Tf)$ and $v$ is a \lsc~\SSup, 
    then
    $$ u(x,t)\leq v(x,t)\quad \hbox{on  }\Omegb\times[0,\Tf).$$
    In the case of strong subsolutions, the result holds for subsolutions which are regular at the
    boundary.  
\end{corollary}

\section{Control problems, stratifications and state-constraints conditions}

\index{Control problem!state-constrained}
In this section we consider finite horizon, deterministic control problems with state-constraints
conditions on the space-time trajectory: $(X(s),T(s))\in\Omegb\times[0,\Tf)$. Here, 
$\Omega$ is a domain in $\R^N$, which is not required to be bounded or regular a priori. 

In order to formulate such problems, let $\Omegb\times[0,\Tf)$ be a stratified domain in
the sense of Definition~\ref{def:stratdom}. We assume that the dynamics, discounts and costs are
defined in $\R^N \times [0,\Tf]$---which is not a loss of generality---and may be discontinuous on
the submanifolds $\Man{k}$ for $k<N+1$, and $\Man{k}_0$ for $k<N$. More precise assumptions will be
given later on.

Following Section~\ref{Gen-DCP}, we first define a general control problem associated to a
differential inclusion. As we mention it above, at this stage, we do not need any particular
assumption concerning the structure of the stratification, nor on the control sets. We also use the
same notations and assumptions as in Section~\ref{Gen-DCP}.

\

\noindent\textsc{The control problem ---} we embed the accumulated cost in the
trajectory by solving a differential inclusion in $\R^N\times\R$, namely \eqref{eq:diff.inc} and we
introduce the value function which is defined only on $\Omegb \times [0,\Tf)$ by
$$U(x,t)=\inf_{\cT(x,t)}\Big\{\int_0^{+\infty} 
     l\big(X(s),T(s)\big)\exp(-D(s))ds\Big\}\;,
$$
where $\cT(x,t)$ stands for all the Lipschitz trajectories $(X,T,D,L)$ of the differential inclusion
which start at $(x,t)\in \Omegb\times [0,\Tf)$ and such that $(X(s),T(s)) \in \Omegb\times [0,\Tf)$
for all $s>0$.

Contrary to Section~\ref{Gen-DCP}, we point out that assumptions are needed in order to have
$\cT(x,t)\neq \emptyset$ for all $(x,t) \in \R^N \times (0,\Tf)$: indeed, while the boundary
$\{t=0\}$ does not pose any problem, there is a priori no reason why trajectories $s\mapsto X(s)$ 
satisfying the constraint to remain in $\Omegb$ for given $(x,t)\in\Omegb\times [0,\Tf)$ should exist.
Therefore, the fact that $\cT(x,t)$ is non-empty will be an assumption in all this part: we will
assume equivalently,

\begin{assumption}{\hyp{\hU}}{the value function $U$ is locally bounded on $\Omegb \times [0,\Tf)$\;.}
    \label{page:hU} 
\end{assumption}

\vspace*{-4mm}

A first standard result gathers Theorem~\ref{DPP} and \ref{SP}\index{Stratified
solutions!state-constrained control problems}\index{Dynamic Programming Principle!for value functions in the state-constraints case}
\begin{theorem}\label{DPP-SP-SC}\emph{--- Dynamic Programming Principle, Supersolution
    Properties.}\smsp 
    Under assumptions \HBCL and \hyp{\hU}, the value function $U$ is \lsc and satisfies
    $$U(x,t)=\inf_{\cT(x,t)}\Big\{\int_0^\theta
    l\big(X(s),T(s)\big)\exp(-D(s)) ds+U \big(X(\theta),T(\theta))\exp(-D(\theta))\big)\Big\}\;,$$
    for any $(x,t)\in\R^N\times [0,\Tf)$, $\theta >0$.
    Moreover, if $\F$ is defined by \eqref{hamil.strat}, then the value function $U$ is 
    a viscosity supersolution of 
    \begin{equation}\label{eq:super.H-SC}
        \F(x,t,U, DU) = 0 \quad \hbox{on  } \Omegb \times[0,\Tf)\; ,
    \end{equation}
    where we recall that $DU=(D_x U, D_t U)$.
\end{theorem}

We point out that, in the same way as Theorem~\ref{DPP} and \ref{SP},  Theorem~\ref{DPP-SP-SC} holds
in a complete general setting, independently of the stratification we may have in mind. The
value function is \lsc as a consequence of the compactness of the trajectories $(X,T,D,L)$.

We conclude this first part by the analogue of Lemma~\ref{lem:super.dpp} showing that supersolutions
always satisfy a super-dynamic programming principle, even in this constrainted setting: again we
remark that this result is independent of the possible discontinuities for the dynamic, discount and
cost.\index{Dynamic Programming Principle!for supersolutions in the state-constraints case}

\begin{lemma}\label{lem:super.dpp-SC}
    Under assumptions \HBCL, \hyp{\hU} and \hyp{\Sub}, if $v$ is a bounded \lsc supersolution of $
    \F(x,t,v, Dv) = 0$ on $\Omegb \times(0,\Tf)$, then for any $(\xb,\tb)\in\Omegb \times(0,\Tf)$
    and any $\sigma >0$,
    \begin{equation}\label{ineq:super.dpp-SC}
        v(\xb,\tb)\geq
        \inf_{\cT(\xb,\tb)}
        \Big\{\int_0^{\sigma}
        l\big(X(s),T(s)\big)\exp(-D(s)) \ds+v\big(X(\sigma),T(\sigma)\big)\exp(-D(\sigma))\Big\}
    \end{equation}
\end{lemma}

\begin{proof} The idea is to use Lemma~\ref{lem:super.dpp} with a penalization type argument.

    To do so, as in the proof of Lemma~\ref{lem:super.dpp}, we are going to prove Inequality
    \eqref{ineq:super.dpp-SC} for fixed $(\xb,\tb)$ and $\sigma$, and to argue in the domain
    $B(\xb,M\sigma)\times [0,\tb]$ where $M$ is given by \hyp{\mathbf{\BCL}}, thus in a bounded
    domain. Next, for $\delta >0$ small, we set 
    $$ v_\delta (x,t):=\begin{cases}
    v(x,t) & \hbox{if $x \in \Omegb$} \\
    \delta^{-1} & \hbox{otherwise}
    \end{cases}$$
    Since we argue in $B(\xb,M\sigma)\times [0,\tb]$, $v_\delta$ is \lsc in $B(\xb,M\sigma)\times [0,\tb]$.

    Next we change $\BCL$ into $\BCL_\delta$ in the following way: if $x\in \Omega$,
    $\BCL_\delta(x,t)=\BCL(x,t)$, while if $x\notin \Omega$, then $(b_\delta, c_\delta, l_\delta)
    \in \BCL_\delta(x,t)$ if \\
    $(a)$ either $(b_\delta, c_\delta,l_\delta)=(b,c,l+\delta^{-1}d(x,\Omegb))$
    where $(b,c,l)\in \BCL(x,t)$ and $d(\cdot,\Omegb)$ denotes the distance to $\Omegb$,\\
    $(b)$ or $(b_\delta, c_\delta, l_\delta)=(0,1,\delta^{-1})$.

    Now, if we set for $(x,t) \in B(\xb,M\sigma)\times [0,\tb]$
    $$\F_\delta(x,t,r,p):=\sup_{(b_\delta ,c_\delta ,l_\delta )\in \BCL_\delta (x,t)}
    \big\{ -b_\delta \cdot p +c_\delta  r -l_\delta  \big\}  \;,$$
    then $v_\delta$ is a \lsc supersolution of $\F_\delta(x,t,v_\delta,Dv_\delta)=0$ in
    $B(\xb,M\sigma)\times (0,\tb)$. Indeed, at the same time $\F_\delta \geq \F$ if $x\in
    \Omegb$ and $\F_\delta(x,t,r,p)\geq r-\delta^{-1}$ if $x\notin \Omega$.

    Therefore Lemma~\ref{lem:super.dpp} implies
    $$v_\delta (\xb,\tb)\geq
        \inf
        \Big\{\int_0^{\sigma}
        l_\delta\big(X_\delta(s),t-s\big)\exp(-D_\delta (s)) \ds+v_\delta 
        \big(X_\delta(\sigma),T_\delta(\sigma)\big)\exp(-D_\delta(\sigma))\Big\}\; ,
    $$
    the infimum being taken on all the solutions $(X_\delta, T_\delta, D_\delta,L_\delta)$ of the
    $\BCL_\delta$ differential inclusion. 

    To conclude the proof, we have to let $\delta$ tend to $0$ in the above inequality where we can
    notice that $v_\delta (\xb,\tb) = v(\xb,\tb)$. To do so, we pick an optimal or $\delta$-optimal
    trajectory $(X_\delta,T_\delta, D_\delta,L_\delta)$.

    By the uniform bounds on  $\dot X_\delta, \dot T_\delta, \dot D_\delta,\dot  L_\delta,$
    Ascoli-Arzela' Theorem implies that up to the extraction of a subsequence, we may assume that
    $X_\delta,T_\delta, D_\delta,L_\delta$ converges uniformly on $[0, \sigma]$ to $(X,T,D,L)$. And
    we may also assume that they derivatives converge in $L^\infty$ weak-* (in particular $\dot
    L^\delta =l^\delta$).

    We use the above property for the $\delta$-optimal trajectory, namely
    $$
        \int_0^{\sigma}
        l_\delta\big(X_\delta(s),t-s\big)\exp(-D_\delta (s)) \ds+v_\delta
        \big(X_\delta(\sigma),T_\delta(\sigma)\big)\exp(-D_\delta(\sigma)) -\delta \leq v(\xb,\tb)\; ,
    $$
    in two ways: first by multiplying by $\delta$, using that $l_\delta\geq -M +
    \delta^{-1}d(x,\Omegb)$ and the definition of $v_\delta$ outside $\Omegb$, we get 
    $$\int_0^{\sigma} d(X_\delta(s),\Omegb)) \exp(-Ms)ds + 
    \1_{X_\delta(\sigma)\notin \Omegb}\exp(-M\sigma)= O(\delta)\; .$$
    Then, the uniform convergence of $X_\delta$ and the fact that both terms in the left-hand side
    necessarily tend to $0$, meaning that $X(s) \in \Omegb$ for any $s \in [0,\sigma]$. And the
    proof is complete. 
\end{proof}

Now we turn to the subsolution properties. We have the following analogue of Theorem~\ref{SubP}
but only in the case of a  stratification which is a \LFS-one on the boundary.
\begin{theorem}\label{SubP-SC} 
    Assume that $\Omegb\times[0,\Tf)$ is a stratified domain which satisfies the \LFS-requirement for
    any point of the boundary and \QRB, and also that \HBASF holds. Then the result of
    Theorem~\ref{SubP} remains valid
    \begin{enumerate}
        \item[$1.$] for $U^*(x,t)$, associated with $\M$ and $(\F^k)_k$ in $\Omegb \times (0,\Tf)$;
        \item[$2.$] for $U^*(x,0)$, associated with $\M_0$ and $(\F^k_0)_k$ in $\Omegb\times \{0\}$. 
    \end{enumerate}
    Moreover $U^*$ is a regular subsolution in the domain and on the boundary both in
    $\Omegb \times (0,\Tf)$ and $\Omegb\times \{0\}$.
\end{theorem}

We do not know if the strange assumption on the stratification on the boundary is necessary or not. But
clearly cusps on the boundary may create some difficulty for the control problem.

\begin{proof} 
    We provide both proofs in the $\Omegb \times (0,\Tf)$-case, the $\{t=0\}$-one being analogous.
    Also, we restrict ourselves to the $\Man{k}$ which are included in the boundary since otherwise
    the proof of Theorem~\ref{SubP} fully applies.

    \smallskip

    \noindent\textbf{(a)} The proof of Theorem~\ref{SubP}-$(i)$, \ie that $U^* = (U|_{\Man{k}})^*$
    on $\Man{k}$ needs to be slighly modified since the trajectories we use have to stay in $\Omegb
    \times (0,\Tf)$. Of course, we also do it in the case of an \AFS\footnote{Below $\Omegb \times (0,\Tf)$
    should be replaced by its image by the diffeomorphism which flattens the stratification but we keep the notation 
    $\Omegb \times (0,\Tf)$ since this does not change anything and this clarifies the important points by avoiding new 
    notations.}.

    In a first step, we can repeat readily the proof of Theorem~\ref{SubP}-$(i)$: if, for $\e$ small
    enough, the trajectory $(\xe,\te)+s b$ exits $\Omegb \times (0,\Tf)$ before time $s_\e$
    (otherwise we are done), this proves at least that $U^*(x,t)=(U|_{\domeg \times
    (0,\Tf)})^*(x,t)$. We choose the minimal integer $l$ such that $(\xe,\te) \in \Man{l}$.

    If $l=k$, we are done hence we may assume \wlg that $l>k$ and repeat exactly the
    same proof: since $b \in V_l$ by the property of the flat stratification, the trajectory
    $(\xe,\te)+s b$ stays in $(\xe,\te)+V_l$ and two cases may occur:\\
    -- either it stays in $\Man{l}$ till time $s_\e$, at least for a subsequence and we are done;\\
    -- or, for $\e$ small enough it leaves $\Man{l}$ at some point which is necessarily in
    $\Man{l'}$ for some $l'<l$. But this would contradict the minimality of $l$ and the proof is
    complete.

    Given this first result, the proof of $(ii)$ follows exactly from the arguments of the proof of
    Theorem~\ref{SubP}.

    \smallskip

    \noindent\textbf{(b)} It remains to show the regularity properties of $U^*$. 
    Because of \QRB, we have just to show \eqref{reg-sub-Hk} which is a consequence of the
    lower-semicontinuity of $U$. Indeed, assume by contradiction that, for some
    $(x,t)\in \Man{k}$, we have $$U^*(x,t) > \limsup\, \{U^*(y,s): (y,s)\in \Man{k+1}\cup \cdots
    \Man{N+1}\} \; .$$ By $(i)$, we have $U^*(x,t)= \lim U(\xe,\te)$ for some sequence
    $((\xe,\te))_\e$ of points of $\Man{k}$ and by the lower-semicontinuity of $U$, there exists
    $(\ye,\se) \in \Man{k+1}\cup \cdots \Man{N+1}$ such that $U(\ye,\se) \geq U(\xe,\te)-\e$. Hence 
    $$ \limsup U^* (\ye,\se) \geq \limsup U(\ye,\se) \geq U^*(x,t)\; ,$$
    proving the claim by contradiction.
\end{proof}

Now we can give the final result.

\begin{theorem}\label{VF-SC-GR}\emph{--- The value function as the unique stratified solution.}\smsp
    Let $\Omegb\times[0,\Tf)$ be a stratified domain which satisfies the \LFS-requirement for any
    point of the boundary and \QRB, and also that \hyp{\hU} holds. Then the value function $U$ is
    continuous and the unique stratified solution of the state-constrained problem in the two
    following cases
    \begin{enumerate}
        \item[$(i)$] \HBASF holds.
        \item[$(ii)$] \HBASFstar and \HBAIDCP hold.    
    \end{enumerate}
\end{theorem}

%
%

\chapter{Classical Boundary Conditions and Stratified Formulation}
\label{chap:RefBF}

\index{Boundary conditions!classical}

\abstract{This chapter answers following question: in which cases are
classical Ishii viscosity (sub)solutions of problems with classical boundary conditions (Dirichlet,
Neumann, mixed, etc.) also stratified solutions of the associated state-constrained problem? The
specific case of the tanker problem is also considered: the stratified formulation is needed in
order to get a well-posed problem.}

In this chapter we investigate the connections between stratified problems with state-constraints
conditions and classical---or almost classical---problems with boundary conditions: Dirichlet,
Neumann, mixed boundary conditions. Of course, the interest of the stratified formulation is to
allow to treat cases where either the boundary is not smooth or the boundary conditions may present
discontinuities, and also both at the same time.

Clearly our aim cannot be to give the most general results: this would be 
unreadable and of a poor interest. But what is done in Section~\ref{sec:solutionTP}
for the Tanker Problem shows that the stratified formulation allows to treat very general
problems, even with exotic ``boundary conditions''. Actually, the reader can notice that, 
in this framework, there is no main difference between the equation and the boundary
conditions. As a consequence, most of the Dirichlet, Neumann, oblique derivatives and mixed
problems we are going to consider have a unique stratified solution provided that we formulate
them in the right way and that the ``natural assumptions''---meaning here essentially
\HBASFstar complemented with \HBAIDCP---are satisfied. We refer to Chapter~\ref{chap:openpb-partV}
below for some discussions on other types of problems, including more interactions between the
equation and the initial conditions as well as stationary problems.

Here we address the following two complementary
questions, mainly in very simple frameworks, whose answers may emphasize the role and the interest
of the stratified formulation:
\begin{enumerate}
\item[$(i)$] in which cases classical Ishii viscosity 
solutions and stratified solutions are the same? Of course, in such cases, the 
theory which is developed in the previous chapter provides complete 
comparison results;
\item[$(ii)$] on the contrary, in which cases is the stratified formulation needed 
because the Ishii formulation is not precise enough to identify the ``good''
solution?
\end{enumerate}

In order to do so and focus on the main difficulties, throughout this Chapter we make several
simplifications and assumptions that we sum up as follows:

\begin{assumption}{\HSBC}{Simplified Framework for Classical Boundary Conditions.}
    \label{page:simplifications}
    We assume that simplifications 1 to 4 below hold. Essentially, this means that $(i)$ we have a
    standard Cauchy problem with a continuous initial data; $(ii)$ the equation has no
    discontinuities inside the domain; $(iii)$ the domain is bounded, associated to a
    time-independent stratification of the boundary and \QRB holds; $(iv)$ we assume that the``good
    framework'' for the stratified approach is satisfied, \ie \HBACP and \HBASFstar hold.
\end{assumption}

Notice that the ``good framework'' assumption is not a simplification, it is mandatory to treat the
problem through the stratified approach. Other than that, the other hypotheses are really simplifications, not
\emph{limitations}: the methods and tools in this book allow to cover far more general situations
and again, we refer to Chapter~\ref{chap:openpb-partV} for some possible generalizations. Let us now
be more precise on \HSBC.

\

\noindent\textbf{Simplification 1 ---} 
We assume that the Hamiltonian has the following structure $\F(x,t,u,(D_xu,D_tu))=u_t +
H(x,t,Du)$ if $x\in \Omega, t \in (0,\Tf)$, \ie we have a Cauchy problem associated to an initial data~$u_0 \in C(\Omegb)$. 
The problem is then written as
\begin{equation}
\begin{cases}\label{standardHJB}
u_t + H(x,t,D_x u)=0 & \hbox{in $ \Omega \times (0,\Tf)$}\;, \\
u(x,0) =u_0 (x) & \hbox{in $\Omega$}\;,
\end{cases}
\end{equation}
where $H$ has the form
\begin{equation}\label{eq:Hmsh}
H(x,t,p_x):= \sup_{\alpha \in A}\,\left\{-b(x,t, \alpha)\cdot p_x -l(x,t,\alpha)\right\}\;,
\end{equation}
for any $x\in \Omegb$, $t\in [0,\Tf)$, $p_x\in \R^N$, where $A$ is a compact metric space. Of
course, for the stratified approach, if $x\in \domeg$ or if $t=0$, $\F$ has to incorporate the terms
corresponding to the boundary and initial conditions.

\

\noindent\textbf{Simplification 2 ---} In problem~\eqref{standardHJB} we restrict ourselves to the
case where the equation inside the domain is continuous. This means that the difficulty only comes 
from the boundary geometry and boundary data. 

So, in this chapter, $b,l$ are continuous functions on $\Omegb\times [0,\Tf) \times A$, taking
values respectively in $\R^N$ and $\R$.
In order to reframe the situation in a stratified setting, let us mention that the notation
$b(x,t,\alpha)$ always refers to the (spatial) dynamic defining $H$; of course, the time dynamic
is $-1$, yielding the $u_t$-term in the equation.
Introducing the set $\BCL$ below, we use the bold notation $(\gb,\gc,\gl)\in\BCL(x,t)$ with of
course here, $\gc=0$. This means that for some control $\alpha$, 
$$\begin{aligned}
   \gb=(\gb^x,\gb^t) &=  (b(x,t,\alpha),-1)\;,\\
   \gl &=  l(x,t,\alpha)\;.
\end{aligned}$$
We also recall that in various computations, we use the notation $p=(p_x,p_t)\in\R^N\times\R$
for the complete gradient variable.

\

\noindent\textbf{Simplification 3 ---} We assume that the geometry of the boundary stratification and
boundary data singularities is time-invariant. This implies that 
$$\Man{N+1}=\Omega \times\R\;,\quad \domeg \times \R=(\tMan{N-1}\cup\cdots \cup \tMan{0})\times\R\;,$$
where $(\tMan{k})_{k=0..(N-1)}$ is a (stationary) stratification of
$\partial\Omega$. Notice in particular that here, the geometry which is induced at time
$t=0$ is not different from the one for positive times.

\

\noindent\textbf{Simplification 4 ---} We assume that $\Omega$ is a bounded domain and that \QRB
holds\footnote{Obviously, \QRB is a natural assumption for Neumann type problems where ``pointing
inward'' or ``pointing outward'' to the domain should have a clear sense.}. In particular, this
allows to forget about localization hypothesis \LOCa / \LOCaEV, but it simplifies also some
arguments especially in the Neumann case. Notice that the equivalence between stratified and Ishii
solutions is purely a local result so that the boundedness of $\Omega$ is not really restrictive of
course.

\

\noindent\textbf{The good framework holds ---} 
In order to make everything work, we need to assume that we are in the ``good framework for
HJ Equations with discontinuities'' by requiring at least \HBASFstar. Of course, this assumption imposes
conditions on both the equation---through $H$ or $b(x,t,\alpha), l(x,t,\alpha)$---and the boundary
condition we are interested in.  Concerning Assumption~\TCBCL inside the domain, it derives
immediately from \HBACP--\HBAHJ and we leave to the reader the checking that, on each different
case, under the hypotheses we make, it will be satisfied up to the boundary.

As expected, \NCBCL just the Hamiltonian $H$ or equivalently the dynamics
$b(x,t,\alpha)$.  Since the stratification does not depend on $t$, then for any $(x,t)\in \Man{k}$,
$T_{(x,t)}\Man{k}=T_{x} \tMan{k}\times \R$ and therefore $(T_{(x,t)}\Man{k})^\bot =(T_{x}
\tMan{k})^\bot \times\{0\}$. Taking into account the regularity of $b$, this allows to express
\NCBCL in a rather simple way, namely: for any $(x,t)\in \Man{k}$, there exists $\delta>0$ such that
\begin{equation}\label{NC:H2} 
    \big\{b(x,t,\alpha); \alpha \in A\big\} \cap (T_{x} \tMan{k})^\bot \supset B(0,\delta) 
    \cap (T_{x} \tMan{k})^\bot\;.
\end{equation}

As a consequence of \NCBCL--\TCBCL and Lemma~\ref{tgfields}, the Hamiltonians $\F^k$ we will define
on each $\Man{k}\subset \domeg \times (0,\Tf)$ satisfy the right assumptions, even if this is not
completely obvious on the formulas which define them. Hence, we will be able to apply
partially the by-now standard tangential regularization procedure to the subsolutions. However, we
point out that the ``$\min$'' in the Ishii subsolution formulation is a non-trivial difficulty when
trying to perform the regularization \emph{up to the boundary}, since the boundary condition does not
satisfy the needed coercivity requirement.

Again we refer the reader to Chapter~\ref{chap:openpb-partV} for extensions to problems where
discontinuities also occur inside $\Omega\times(0,\Tf)$. Clearly some of these extensions are easy
using some ideas of this chapter: typically, if the discontinuities of $H$ inside $\Omega \times
[0,\Tf)$ stay away from the boundary; but some other ones are more delicate, if the discontinuities
of $H$ inside $\Omega \times [0,\Tf)$ interfere with the boundary.

Apart from \HSBC, we will use some other assumptions: \IDP for the Dirichlet problem and
and \Hgamma---to be introduced later---for the oblique derivative problem.

\section{On the Dirichlet problem}\label{RefBF:Dir}
\index{Boundary conditions!Dirichlet}

We are interested in this section in the Dirichlet problem for Hamilton-Jacobi-Bellman Equations,
namely \eqref{standardHJB} associated with the boundary condition
\begin{equation}\label{DP}
u(x,t)=\varphi(x,t) \quad\hbox{on $\domeg\times (0,\Tf)$\;,}
\end{equation}
where we first assume that $\varphi$ is a continuous function which satisfy the compatibility
condition
\begin{equation}\label{eqn:compcondid}
u_0(x)=\varphi(x,0)\quad\hbox{on  }\domeg\; .
\end{equation}

In this classical case,  there are two kinds of results which are described in
the book \cite{Ba} and are originated from the works of Perthame and the first
author \cite{BP1,BP2,BP3}.

\begin{enumerate}
    \item[$(a)$] The \emph{discontinuous approach} where one tries to determine the
        minimal and maximal solution of \eqref{standardHJB}-\eqref{DP} in full generality. By this
        we mean here: without any particular additional assumption on the dynamic and cost, and
        without assuming the boundary of $\Omega$ to be smooth. The result is that there exist a
        minimal solution $U^-$ and a maximal solution $U^+$ which are value functions of exit time
        problems,  $U^-$ being associated to the best stopping time on the boundary, while $U^+$ is
        associated to the worst stopping time on the boundary. 
    
    \item[$(b)$] The \emph{continuous approach} in which one looks for conditions
        under which the value function is continuous and the unique solution of
        \eqref{standardHJB}-\eqref{DP}. In \cite{BP3}, the result is obtained under classical
        assumptions on the dynamics and cost, plus an hypothesis of normal controllability on the
        boundary which looks very much like \NCe. This second type of results require some
        regularity of the boundary, $C^{1,1}$ in general.  
\end{enumerate}

As we said, in this section our aim is to reformulate the Dirichlet problem in the
\emph{stratified framework}, in order to investigate the cases when it is
equivalent to the classical viscosity solutions formulation and then to examine
the type of extensions that we can get in that way. 

We recall that, in order to avoid confusions, we use bold faces for the $\BCL$ elements while $b,l$
are the ones defining $H$, and $p=(p_x,p_t)$ is the gradient. There are also fundamental assumptions
and several simplifications that we assume, listed on page~\pageref{page:simplifications}, referred to
as \HSBC.

\subsection{Stratified formulation of the classical case}
\index{Dirichlet problems!stratified formulation}

Reformulating the problem is quite clear and classical: if $(x,t)\in \Omegb\times [0,\Tf]$, we set
$$\BCLeq (x,t):=\Big\{\, \big((b(x,t,\alpha),-1),0,l(x,t,\alpha)\big);\ \alpha \in A\,\Big\}\; ,$$
``eq'' for ``equation'' and, if $(x,t)\in \domeg\times [0,\Tf]$, we introduce $$\BCLbc
(x,t):=\Big\{\,\big((0,0), 1,\varphi(x,t)\big)\,\Big\}\; ,$$ ``bc'' for ``boundary condition''.
Indeed, at the level of the general Hamiltonian $\F$, this produces the expected term on the
boundary, namely $$ - \gb \cdot p + \gc u-\gl= u-\varphi(x,t)\; ,$$ and, for the control point of
view, this provides a  $0$-dynamic allowing to stop at the point $(x,t)$ and pay a cost which is
$\varphi(x,t)$, the discount factor being $1$.

Of course the complete $\BCL$ is given by $$\BCL(x,t)= \begin{cases}\qquad \BCLeq (x,t) &\text{if }
(x,t)\in \Omega\times (0,\Tf]\;,\\[2mm] \overline{\mathrm{co}}\Big(\BCLeq (x,t)\cup \BCLbc
(x,t)\Big) &\text{if }(x,t)\in \domeg\times (0,\Tf]\;.  \end{cases}$$ At $t=0$, and this is by now
classical in this book, we need to add to $\BCL$ the term $(\gb,\gc,\gl)=((0,0),1,u_0(x))$ in order
to take into account the initial data.  Then, $\BCL(x,0)$ is given by

$$\BCL(x,0)= \begin{cases} \overline{\mathrm{co}}\Big(\BCLeq (x,0)\cup
\big\{((0,0),1,u_0(x))\big\}\Big)\;, &\text{if } x\in \Omega\;,\\[2mm]
\overline{\mathrm{co}}\Big(\BCLeq (x,0)\cup \BCLbc(x,0)\cup \big\{((0,0),1,u_0(x))\big\}\Big)\;,
&\text{if } x\in \partial\Omega\;.  \end{cases}$$ With this point of view, we end up with just a
state-constrained problem since the trajectory $(X,T)$ exists for all times and stays in
$\Omegb\times [0,\Tf]$, the Dirichlet condition allowing the choice $\gb=0$ on the boundary. This is
also the case for the initial data at $t=0$:  $\gb^t=-1$ for any $(\gb,\gc,\gl)$ in $\BCLeq (x,0)$
for which $\gb\neq 0$ but the initial data term, namely $((0,0),1,u_0(x))$ allows to stay in
$\Omegb\times [0,\Tf]$ and actually $$ \F_{init}(x,r,p_x)= r-u_0(x)\quad\hbox{for  }x\in \Omegb,\
r\in \R,\  p_x \in \R^N\; ,$$ because of the compatibility condition \eqref{eqn:compcondid}.

Since $H$ is continuous on $\Omegb\times [0,\Tf]$, the stratified approach consists in considering,
for $t>0$, the stratification $\Man{N+1}=\Omega \times (0,\Tf)$ and $\Man{N}=\domeg\times (0,\Tf)$.
In order to apply the above results, we have to impose at least two conditions.  \begin{enumerate}
    \item[$(i)$] Some regularity of $\domeg$. Here, $C^{1,1}$---exactly as in \cite{BP3}---is
        natural in general since we have to flatten $\Man{N}$ while keeping the needed properties on
        $H$, in particular \TC. But this can be reduced to $C^1$ if $H$ is coercive, to the cost of
        sophisticating a little bit our arguments, treating differently the variables $t$~and~$x$.
    \item[$(ii)$] Some normal controllability assumptions which turn out to be also the same as in
        \cite{BP3}, namely for any $(x,t) \in \domeg \times [0,\Tf]$, the existence of two controls
        $\alpha_i=\alpha_i(x,t)$ for $i=1,2$ such that \begin{equation}\label{eq:nca:dir}
        b(x,t,\alpha_1)\cdot n(x)<0\quad,\quad  b(x,t,\alpha_2)\cdot n(x)>0\; , \end{equation} where
$n(x)$ is the unit outward normal vector to $\domeg$ at $x$.  \end{enumerate} We come back later on
the advantages of this new approach but let us examine first the boundary condition from the
stratified point of view.

\medskip

\noindent\textsc{Computing the boundary condition ---} On the boundary $\domeg \times (0,\Tf]$,
$\BCL$ is obtained by considering the convex enveloppe of elements of the form
$$(\gb,\gc,\gl)=\big(\,(b(x,t,\alpha),-1),0,l(x,t,\alpha)\,\big)\in\BCLeq(x,t)\;,$$ associated to
Hamiltonian $H$, and of $\big(\,(0,0), 1,\varphi(x,t)\,\big)$ associated to the Dirichlet boundary
condition. Therefore, we have to consider all the combinations $$\Big(\,\mu
(b(x,t,\alpha),-1)\;,\,(1-\mu)\;,\,\mu l(x,t,\alpha)+(1-\mu)\varphi(x,t)\,\Big )$$ where $0\leq
\mu\leq 1$ satisfies $\mu \gb=\mu (b(x,t,\alpha),-1) \in T_{(x,t)}\Man{N}\;,$ in other words
$b(x,t,\alpha) \in T_x \domeg$.

In order to compute $\F^N(x,t,r,p)$, we look at the supremum in $\mu$ and $(\gb,0,\gl)\in\BCL(x,t)$
with $\gb^x\in T_x \domeg$, of $$- \mu \gb\cdot p +(1-\mu)u - (\mu \gl+(1-\mu) \varphi(x,t))= \mu
(-\gb\cdot p -  \gl)+(1- \mu)(u - \varphi(x,t))\;.  $$ Clearly, this supremum is achieved either for
$\mu=0$, or $\mu=1$ since the dependence with respect to $\mu$ is linear. Hence, the subsolution
inequality takes the form

\begin{equation}\label{StratDir-MN} \begin{cases} \max\big(u_t + H^{N}(x,t,D_xu), u -
\varphi(x,t)\big)\leq 0 \quad\hbox{on }\Man{N}\;,\\[2mm] \text{where} \\[2mm]
H^{N}(x,t,p_x)=\sup\limits_{b(x,t,\alpha) \in T_x \domeg}\{-b(x,t,\alpha) \cdot p_x -l(x,t,\alpha)
\}\;.  \end{cases} \end{equation}

To the best of our knowledge, this quite unusual inequality never appears in the study of Dirichlet
boundary conditions for HJ-Equations, the closest being the one introduced for state-constrained
problems by Ishii and Koike \cite{IK} but where their Hamiltonian $H_{in}$ takes also into account
inner dynamics.  But, on the other hand, it is rather natural from the control point of view: the
inequality $u_t + H^{N}(x,t,D_xu)\leq 0$ means that tangential dynamics are sub-optimal and, in the
same way, the inequality $u - \varphi(x,t) \leq 0$ reflects the sub-optimality of the strategy
consisting in stopping at $(x,t)$, paying the cost $\varphi(x,t)$.  We point out anyway that the
normal controllability plays a role here: such stopping strategy is available to the controller as
soon as the state $(X,T)$ comes close to $\domeg \times (0,\Tf)$ since he can choose to quickly exit
the domain via a dynamic pointing outward $\Omega \times (0,\Tf)$. Then, when $(X,T)$ is on $\domeg
\times (0,\Tf)$, he can either stop and pay the $\varphi$-cost or continue on the boundary using
tangential dynamics $b(x,t,\alpha) \in T_x \domeg$, waiting a better stopping time on $\domeg \times
(0,\Tf)$.

It is also worth remarking that the non-tangential dynamics are taken into account in the Ishii
viscosity subsolution inequality $$ \min\big(u_t + H(x,t,D_xu), u - \varphi(x,t)\big)\leq 0 \quad
\hbox{on  }\domeg\times (0,\Tf)\; .$$

Now we turn to the first key question: do \emph{classical viscosity subsolutions} always satisfy
such $H^N$-inequality in the stratified framework? And, in the case of a less regular boundary---but
still in a stratified framework---, does an analogous one hold on $\Man{k}$ for $1\leq k \leq N$?

\medskip

\noindent\textsc{About the initial condition ---} As we have seen it in the study of stratified
solutions for the state-constrained problem, the way the initial data is taken into account is
important and the points of $\domeg\times \{0\}$ create a difficulty.  Here this difficulty comes
from the interference between the initial data $u_0$ and the Dirichlet boundary condition $\varphi$.

To discuss this difficulty, we first provide the 
\begin{proposition}\label{eqn:ibd-Ishii}\emph{--- The classical Dirichlet problem.}\smsp 
    Assume
    that $H\in C(\Omegb\times [0,\Tf)\times \R^N)$, $u_0\in C(\Omegb)$ and $\varphi$ is a locally
    bounded function on $\domeg\times [0,\Tf)$. If $u$ and $v$ are respectively an \usc classical
    viscosity subsolution and a \lsc classical viscosity supersolution of the Dirichlet problem
    \eqref{standardHJB}-\eqref{DP} then we have $$ u(x,0) \leq u_0(x) \leq v(x,0) \quad \hbox{in
    }\Omega\; ,$$ and $$ \begin{cases} u(x,0) \leq \max\big(u_0(x),\varphi^*(x,0)\big)\quad \hbox{on
        }\domeg\;,\\[2mm] v(x,0) \geq \min\big(u_0(x),\varphi_*(x,0)\big) \quad \hbox{on }\domeg\;.
    \end{cases}$$ 
\end{proposition}

We first want to emphasize the fact that this result holds without any assumption on the smoothness
    of the boundary. Therefore we will always be able to use it, for any type of domain.

We skip the proof of this proposition which is easy, following similar argument as those of
    Section~\ref{IC-HJB}: indeed it suffices to look at either maximum points of $$ (y,s) \mapsto
    u(y,s) -\frac{|y-x|^2}{\e^2}-C_\e s\; ,$$ or minimum points of $$ (y,s) \mapsto v(y,s)
    +\frac{|y-x|^2}{\e^2}+C_\e s\; ,$$ where $0<\e \ll 1$ is a parameter devoted to tend to $0$ and
    $C_\e\gg \e^{-1}$ is a large enough constant.

This result shows that there two main cases \begin{enumerate} \item[$(i)$] The case when $\varphi$ is
            continuous, at least at points of $\domeg\times \{0\}$ and \eqref{eqn:compcondid} holds,
            which implies $$ u(x,0) \leq u_0(x) \leq v(x,0) \quad \hbox{on }\Omegb\;.$$ We are then
            in the situation where \HBAIDCP holds and the stratified approach just requires
        \HBASFstar, which is here given by \HBACP, providing the tangential continuity, and
    \eqref{NC:H2} for the normal controllability.  \item[$(ii)$] If we are not in the first case, a
        further discussion is needed, even if $\varphi$ is still a continuous function. Indeed, as
        we already mention it in Remark~\ref{rem:pbidbord}, if $\varphi$ is continuous but satisfies
        $\varphi(x_0,0)<u_0(x_0)$ at some point $x_0 \in \domeg$, we may easily build a control
        problem---even with full controllability properties---for which the value function is
        discontinuous, and therefore no comparison result can hold. 

        In fact, since the exit cost $\varphi$ is strictly below $u_0$ in a neighborhood of $x_0$,
        the controller aims at paying the cost $\varphi$ but this is possible only by starting from
        a point $(y,s)$ for which $y$ is sufficiently close to the boundary, measured in terms of
        $s$. Hence, the discontinuity of the value function separates the regions where we can
        actually exit from $\domeg$ from the ones where it is impossible.

        This short and maybe vague analysis shows that a natural assumption should be
        \begin{equation}\label{eqn:natassid} u_0(x) \leq \varphi_*(x,0)\quad \hbox{for all  }x \in
        \domeg\;, \end{equation} leading to $v(x,0) \geq u_0(x)$ for all  $x \in \domeg$ by
        Proposition~\ref{eqn:ibd-Ishii}. Unfortunately, we still only get $u(x,0) \leq
\max(u_0(x),\varphi^*(x,0))$ for the subsolution.  \end{enumerate}

In the next sections, we are going to show how to treat these two cases.

\subsection{Continuous data with a stratified boundary}
\index{Dirichlet problems!stratified formulation}

In this section, we address all the above questions in a full generality under assumptions \HSBC,
    given at the beginning of the chapter: in particular $H$ is continuous but $\domeg\times
    (0,\Tf)$ is not smooth anymore.

We point out that the stratification of the boundary is not a priori related to discontinuities in
    the equation/boundary condition, but is just an assumption on the kind of non-smooth boundary we
    can handle. For such domains, we first consider here the case when $\varphi$ is continuous on
    $\domeg \times [0,\Tf)$.

The next subsection will focus on the cases when $\varphi$ may be discontinuous but in a way which
    is compatible with the stratification of the boundary, \ie $\varphi$ is continuous on each
    $\Man{k}$ for $1 \leq k\leq N$.

We first examine the situation on $\Man{N}$ since this is a common denominator of all the cases we
are going to consider.  \begin{proposition}\label{Dir-MN} Assume that \HSBC holds and that $\varphi$
is continuous on $\Man{N}$. Then, \begin{enumerate} \item[$(i)$] if $u$ is any \usc viscosity
            subsolution of the Dirichlet problem, $u\leq \varphi$ on $\Man{N}$; \item[$(ii)$] if
                $\tilde u$ is given by $\tilde u=u$ in $\Omega \times (0,\Tf)$ and for
                $(x,t)\in\Man{N}$, $$ \tilde u (x,t)= \limsup_{\displaystyle \mathop{\scriptstyle
                (y,s) \to (x,t)}_{\scriptstyle  y\in \Omega}} u(y,s)\; ,$$ then $\tilde u$ is still
                a classical viscosity subsolution of the Dirichlet problem and
                \begin{equation}\label{StratDir-MN-bis} \max\big(\tilde u_t + H^{N}(x,t,D_x\tilde
                    u), \tilde u - \varphi(x,t)\big)\leq 0 \quad\hbox{on }\Man{N}\;.  \end{equation}
Moreover $\tilde u$ is regular on $\Man{N}$.  \end{enumerate} \end{proposition}

We first want to point out that, in Proposition~\ref{Dir-MN}, the essential consequence of the
    ``good framework for stratified problem'' is that \eqref{eq:nca:dir} holds on $\Man{N}$.

The introduction of the function $\tilde u$ in order to redefine $u$ on the boundary is classical:
    in fact, it is needed because the viscosity subsolution inequality is not strong enough to avoid
    artificial values of $u$ on the boundary. Indeed since the viscosity subsolution property is
    ensured by the fact that $u\leq \varphi $ on $\Man{N}$, $u$ could be changed on the boundary
    into any \usc function which lies above $u$ and below $\varphi$ on $\Man{N}$, with no link
    whatsoever with the values inside $\Omega \times (0,\Tf)$. The introduction of $\tilde u$
    consists in imposing the ``natural'' values of the subsolution on $\Man{N}$ since they are
    consistent with those in $\Omega \times (0,\Tf)$.  Once this ``cleaning'' of the boundary values
    is done, then we have the desired result, namely that viscosity subsolutions are stratified
    subsolutions on $\Man{N}$.

A different point of view is the regularity of subsolutions on the boundary: we have insisted, since
    the beginning of Part~\ref{S-BC}, that this is a key difficulty in state-constrained problems.
    Here we face it and Proposition~\ref{Dir-MN} solves it in the case when $\Man{N}=\domeg \times
    (0,\Tf)$ by replacing the non-regular subsolution by a regular one, changing only its values on
    the boundary.

In fact, Proposition~\ref{Dir-MN} is a first step for applying inductively Lemma~\ref{lem:rbv}:
    \begin{corollary}\label{Dir-cont-comp}\emph{--- Comparison in the stratified, continuous
        case.}\smsp
            Assume that \HSBC holds and that $\varphi$ is continuous on
$\domeg\times [0,\Tf)$.  \begin{enumerate} \item[$(i)$] The \usc function $\tilde u$ defined
            inductively by the process of Lemma~\ref{lem:rbv} is given by $$ \tilde u(x,t) :=
            \limsup\Big\{u(y,s);\ (y,s)\to (x,t), \ (y,s) \in \Omega \times (0,\Tf) \Big\} \quad
            \hbox{on  }\Man{k}\; , $$ for $1\leq k\leq N$ and $\tilde u$ is still a classical
            viscosity subsolution of the Dirichlet problem and is also a stratified subsolution of
            the state constraint problem which is regular on the boundary.

        \item[$(ii)$] As a consequence, if $u$ and $v$ are respectively classical viscosity sub and
        supersolution of the Dirichlet problem then $$ \tilde u \leq v \quad \hbox{on  }\Omegb
    \times [0,\Tf)\; ,$$ and in particular $$u \leq v \quad \hbox{on  }\Omega \times [0,\Tf)\;.$$
\item[$(iii)$] Thus, there exists a unique viscosity solution of the Dirichlet problem which is
continuous on $\Omegb \times [0,\Tf)$, uniqueness being understood as up to a redefinition of the
solution on the boundary.  \end{enumerate} 
\end{corollary}

The interest of this result is clear: by redefining a classical viscosity subsolution $u$ on the
    boundary, we can show that it becomes a regular \wSSub of a state-constrained problem and
    therefore we can apply the comparison result for stratified sub and supersolutions, which shows
    the uniqueness solutions of Dirichlet problems for general non-smooth boundaries. We insist
    anyway on the fact that, as it was the case in \cite{BP1,BP2,BP3}, the solution is really unique
    only in $\Omega \times [0,\Tf)$ since we have to modify its values on the boundary.

The proof of this result  consists in applying readily Lemma~\ref{lem:rbv}: the definition of
    $\tilde u$ together with the continuity of $\varphi$ imply that $\tilde u$ is still a classical
    viscosity subsolution of the Dirichlet problem since we have $\tilde u \leq \varphi$ on
    $\domeg\times [0,\Tf)$. On the other hand, the \wSSub-property follows directly from our
    assumptions which allows to apply Lemma~\ref{lem:rbv}. Finally the formula for $\tilde u$ given
    in the statement of Corollary~\ref{Dir-cont-comp} comes from similar arguments as the ones used
    in the proof of Corollary~\ref{cor:reg-MN}. We leave the easy checking of all these details to
    the reader.

Now we turn to the \begin{proof}[Proof of Proposition~\ref{Dir-MN}] Let $u$ be an \usc viscosity
    subsolution. 

    \medskip

    \noindent\textbf{(a)} We start proving that $u\leq \varphi$ on $\Man{N}$. We can argue locally
    and therefore assume that $\domeg\times (0,\Tf)=\Man{N}$ is smooth, hence $\domeg$ is smooth. If
    $d$ denotes the distance to $\domeg$, $d$ is at least $C^1$ and we recall that $D_x d(y)=-n(y)$
    if $y \in \domeg$, where $n(y)$ is the unit outward normal to $\domeg$ at $y$. 

    If $(x,t) \in \Man{N}$, we consider the function $$(y,s) \mapsto
    u(y,s)-\frac{(s-t)^2}{\e^2}-\frac{|y-x|^2}{\e^2}-C_\e d(y)\; ,$$ where $C_\e>0$ is a large
    constant to be chosen later. This function has a maximum point $(\ye,\te)$ near $(x,t)$ and, by
    classical arguments, we have $(\ye,\te)\to (x,t)$ and $u(\ye,\te)\to u(x,t)$. 

    We claim that for $C_\e$ large enough, $y_\e\in\partial\Omega$ and
    $u(\ye,\te)\leq\varphi(\ye,\te)$. Indeed, otherwise the $H$-inequality holds at $(\ye,\te)$,
    which leads to $$ \frac{2(\se-t)}{\e^2}+H\Big(\ye,\se, \frac{2(\ye-x)}{\e^2}+C_\e
    D_xd(\ye)\Big)\leq 0\;.$$ But $D_xd (\ye)=D_x d(x) + o(1)=-n(x)+o(1) $ and, by the normal
    controllability assumption---\cf \eqref{eq:nca:dir}---this inequality cannot hold for $C_\e$
    large enough. As a consequence, the claim holds. Then, letting $\e \to 0$, with a suitable
    $C_\e$, we obtain the desired result, using that $\varphi$ is continuous on $\Man{N}$.

    \medskip

    \noindent\textbf{(b)} As we already mentioned it above, the viscosity subsolution inequality
    being reduced to $u\leq \varphi$ on $\Man{N}$, since $\tilde u \leq u$ (because $u$ is \usc), it
    follows that $\tilde u$ is also a viscosity subsolution of the Dirichlet problem.

    Next we have to show that the $\F^N$-inequality holds for $\tilde u$. We may assume without loss
    of generality that $\tilde u$ is Lipschitz continuous because we can perform the regularization
    in the tangent variables (including $t$), and then use the normal controllability property. In
    the same way, we can assume that the boundary is flat and use the definition of $H^N$ not only
    when $x\in \domeg$ but also for $x\in \Omega$. We notice that $H^N(x,t,p_x) \leq H(x,t,p_x)$ if
    $x\in \domeg$ since the supremum is taken on a smaller set than $\BCL$ and, if $n$ is the unit
    outward normal to $\domeg$ (which is flat), $H^N(x-\e n, t,p_x) \to H^N(x, t,p_x)$ when $\e \to
    0$ as a consequence on the \HBACP-assumptions and of the normal controllability.  Hence
    $H^N(x-\e n, t,p_x)\leq H(x,t,p_x) + o_\e(1)$ where the $o_\e(1)$ is uniform for bounded $p$.

    As in Proposition~\ref{IequalS}, it is clear that $\tilde u^\e (x,t):=\tilde u (x-\e n,t)$ is a
    subsolution of $$ \tilde u^\e_t + H^{N}(x-\e n ,t,D_x\tilde u^\e)\leq o_\e(1) \quad\hbox{on
    }\Man{N}\; ,$$ and passing to the limit by a standard stability result (since $\tilde u^\e $
    converges to $\tilde u$ uniformly and since the $o_\e(1)$ is uniform for bounded $p$), we obtain
    \eqref{StratDir-MN-bis}.
    
    Finally $\tilde u$ is regular on $\Man{N}$ by its very definition; and the proof is complete.
\end{proof}

\begin{remark}\label{MgMk} In the above proof, the inequality $u(x,t) \leq \varphi(x,t)$ plays a key
    role.  In fact, even if $\varphi$ is discontinuous, the inequality $u(x,t) \leq \varphi^*(x,t)$
    (with the \usc enveloppe of $\varphi$ on $\partial\Omega\times(0,\Tf)$ of course) can be proved
    not only for points in $\Man{N}$ but for any point where the exterior sphere condition holds,
    \ie there exists $\xb\in \R^N, \bar r >0$ such that $$\overline{B(\xb, \bar r)}\cap \overline
    \Omega=\{x\}\; .$$ The modification consists in reproducing the same proof replacing the
    function $d(y)$ by $\chi(y):= |y-\xb |-\bar r$. Indeed, if $x$ is a minimum point of $\chi$ on
    $\domeg$ and therefore, if $(x,t)\in \Man{k}$,  $(D\chi(x),0)$  is orthogonal to
    $T_{(x,t)}\Man{k}$, allowing to use \NCBCL.  This inequality is therefore a general fact, but
    unfortunately not convenient for the stratification formulation which requires the more
    restrictive inequality $u(x,t) \leq\varphi_*(x,t)$.\\[2mm] \end{remark}

\subsection{Discontinuous data well-adapted to a stratified boundary} 

\index{Dirichlet problems!well-adapted boundary conditions}
In order to go further, \ie to take into account more general boundary conditions $\varphi$, we
introduce the 
\begin{definition}\emph{--- Well-adapted boundary conditions.}\smsp
    Assume that $\overline \Omega \times [0,\Tf)$ is a stratified
domain and let $\varphi : \domeg \times [0,\Tf) \to \R$ be a lower-semicontinuous function.
\begin{enumerate} \item[$(i)$]We say that $\varphi$ is adapted to the stratification if $$\hbox{for
        all  } 1\leq k\leq N\;,\ \varphi|_{\Man{k}} \text{ is continuous.}$$ \item[$(ii)$] Moreover,
            $\varphi$ is said to be W-adapted (``well-adapted'') to the stratification if in
            addition, for any $(x,t) \in \Man{k}$ and any $1\leq k\leq N$, $$
\varphi(x,t)=\liminf_{\displaystyle  \mathop{\scriptstyle (y,s) \to (x,t)}_{\scriptstyle  (y,s) \in
\Man{N}}} \varphi (y,s)\; .$$ 
\end{enumerate} 
\end{definition}

On the other hand, we notice that while hypothesis \HBASFstar ensures some controllability on the
    boundary $\partial\Omega$ for positive times, it is not assumed to be uniform as $t\to0$. In the
    case of discontinuous data, this degeneracy causes some issues with the condition on
    $\partial\Omega\times\{0\}$. In order to avoid that, let us introduce the following assumption

\begin{assumption}{\IDP}{Inward-pointing Dynamic Property.}\label{page:IDP}
For any $x\in \domeg$,
    there exists $\tau,r>0$ and a $C^1$-function $\phi$ defined in $B(x,r)$ such that $\phi(y) = 0$
    if $y \in \domeg\cap B(x,r)$, $\phi(y)>0$ if $y \in \Omega\cap B(x,r)$, satisfying $$\hbox{For
    all  }(y,s)\in \big(\Omega\cap B(x,r)\big)\times[0,\tau]\;,\quad \sup_{\alpha \in A}\left\{
b(y,s,\alpha)\cdot D_x\phi (y)\right\} \geq 0\;.$$ 
\end{assumption}

\vspace*{-8mm}

\begin{figure}[!h] 
    \begin{center}
    \includegraphics[width=0.5\textwidth]{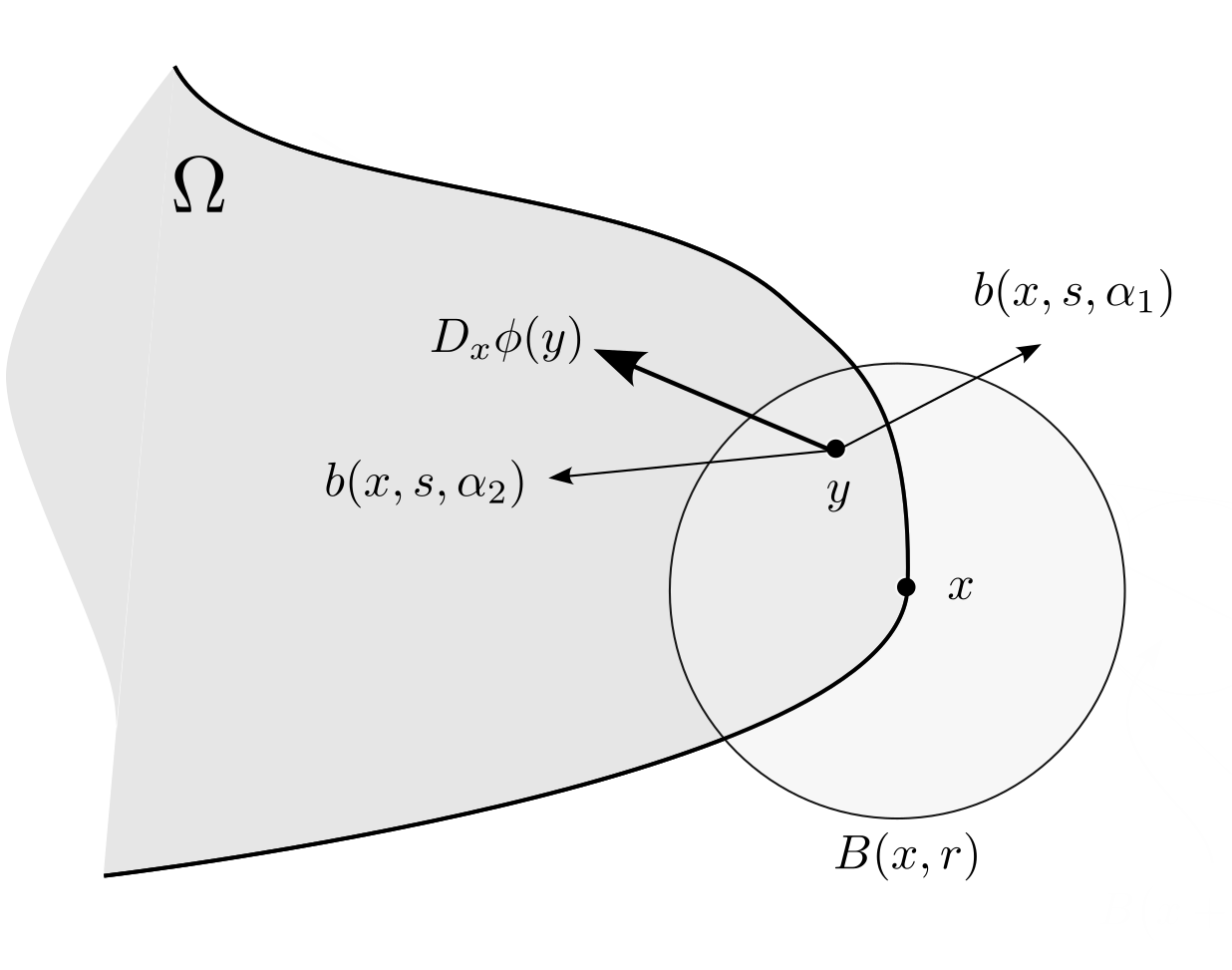}
    \caption{The IDP assumption at time $s$} \label{fig:idp}    
    \end{center}
\end{figure}

Roughly speaking, \IDP means that on a neighborhood of $\partial\Omega\times\{0\}$, $\BCL$ contains
    at least some inward-pointing dynamics $b(y,s,\alpha)$. Indeed, the function $\phi$ above can be
    seen as a local substitute for the distance function, and the condition $b(y,s,\alpha) \cdot
    D_x\phi\geq0$ which is satisfied by at least one control $\alpha$, means that $b(y,s,\alpha)$ is
    pointing inside $\Omega$, at least in a weak sense: while $D_x\phi(y)$ is really pointing
    inwards (or could be tangential at most) on $\domeg$, there is some room for other directions
    inside $\Omega$ which allow $b(y,s,\alpha)$ to point outwards while still satisfying
    $b(y,s,\alpha)\cdot D_x\phi\geq0$, \cf Figure~\ref{fig:idp} below.\footnote{on this figure the
    boundary is smooth but of course more complex, non-smooth situations are allowed here.}

Notice also that if $b(y,s,\alpha)=0$ can be used at $t=0$, it may not be a usable dynamic for
    positive times, so that \IDP is not a trivial hypothesis. We refer the reader to
    Example~\ref{ex:square} and Remark~\ref{rem:idp} below for some example and comments on the
    existence of such pseudo-distance functions $\phi$.

The result for W-adapted boundary conditions is the following
\begin{proposition}\label{Dir-nonreg}\emph{--- Comparison for well-adapted boundary
    conditions.}\smsp
    Assume that \HSBC holds.  
\begin{enumerate} 
    \item[$(i)$] If
            $\varphi : \domeg \times [0,\Tf) \to \R$ is a lower-semicontinuous function, W-adapted
            to the stratification and $u$ is viscosity subsolution of the Dirichlet problem, then
            $\tilde u : \overline \Omega \times [0,\Tf) \to \R$ defined by $\tilde u (x,t)=u(x,t)$
            if $x\in \Omega$ and $$ \tilde u (x,t)= \limsup_{\displaystyle  \mathop{\scriptstyle
            (y,s) \to (x,t)}_{\scriptstyle y\in \Omega}} u(y,s)\quad \hbox{if  } x \in \domeg\; ,$$
            is a regular stratified subsolution of the problem on $\Omegb \times (0,\Tf)$.

    \item[$(ii)$] If in addition we assume that \IDP holds and $u_0 \in C(\overline \Omega)$
        satisfies $ u_0(x)\leq \varphi_*(x,0)$ on  $\domeg,$ then for any viscosity supersolution of
        the Dirichlet problem, $$ \tilde u \leq v \quad \hbox{on  }\overline \Omega \times [0,\Tf)
        \; .$$ In particular, in this case there exists a unique continuous viscosity solution of
        the Dirichlet problem, up to a modification of its values on the boundary.  \end{enumerate}
\end{proposition}

The first part of this proposition says that, under suitable ``standard'' assumptions and
modification of the subsolution on the boundary, then Ishii viscosity subsolutions and stratified
subsolution are the same. For a complete application of this first result, one needs to treat the
initial data and, as it will be clear in the proof, the additional conditions in $(ii)$ imply that
\HBAIDCP is satisfied, since $$\tilde u (x,0) \leq u_0 (x)\leq v(x,0)\quad \hbox{on  } \overline
\Omega\;.$$ Notice that this double inequality prevents maximum points of $u-v$ to be achieved on
$\domeg \times \{0\}$ if this maximum is assumed to be strictly positive.

\begin{proof}
    The proof of $(i)$ still consists in applying Lemma~\ref{lem:rbv} by
    induction, using of course that \QRB holds in order to have to consider only one connected
    component (locally speaking).

    \smallskip

    \noindent\textbf{(a)} 
    The first step is easy: by Proposition~\ref{Dir-MN}, $\tilde u$ is
    a stratified subsolution on $\Man{N}$ and it remains to show that the same is true on any
    $\Man{k}$.
  
    The main difficulty is to show that $\tilde u\leq \varphi$ on $\Man{k}$ for any $k$. If $(x,t)
    \in \Man{k}$, we use a tangential regularization of $\tilde u$ in a neighborhood of $(x,t)$,
    \cf Proposition~\ref{reg-by-sc}. We obtain Lipschitz continuous functions $\tilde u^\e$ which
    lie below $\varphi$ on each connected component of $\Man{N}$. Therefore, $\tilde u^\e (x,t) \leq
    \varphi(x,t)$ on $\Man{k}$ since the lower semicontinuous enveloppe of $\varphi$ can be computed
    using only points of $\Man{N}$. Passing to the limit as $\e\to0$ yields the desired inequality,
    $\tilde u\leq\varphi$ on $\Man{k}$.
  
    Once we have this inequality, the $\F^k$-one comes by applying Lemma~\ref{lem:rbv} by induction.

    \smallskip

    \noindent\textbf{(b)} For the comparison result, the only additional difficulty is $t=0$ and
    more precisely the points of $\domeg \times \{0\}$ where we have to show that $\tilde u \leq
    u_0$ and $v\geq u_0$. The proof for $v$ is easy since Proposition~\ref{eqn:ibd-Ishii} implies $v
    \geq \max(u_0,\varphi_*)\geq u_0$ on $\domeg \times \{0\}$.

    But for the subsolution case, we only get $\tilde u \leq \max(u_0,\varphi^*)$ on $\domeg \times
    \{0\}$, which is clearly not sufficient. To turn around this difficulty at $(x,0)$, $x\in
    \domeg$, we introduce the function $$(y,s) \mapsto \tilde
    u(y,s)-\frac{s}{\e}-\frac{|y-x|^2}{\e}-\frac{\alpha}{\phi(y)}\; ,$$ where $0<\alpha\ll \e \ll 1$
    are parameters devoted to tend to $0$ and $\phi$ is the function coming from assumption \IDP at
    $x$.

    By classical arguments, this function has a local maximum point $(\ye,\se)$ in a neighbordhood
    of $(x,0)$ and $(\ye,\se) \to (x,0)$ with $\tilde u(\ye,\se) \to \tilde u(x,0)$ at least if
    $\alpha, \e \to 0$ with $\alpha\ll \e$ \footnote{By the definition of $\tilde u$, the values of
    $\tilde u$ on the boundary are the limits of the values of $\tilde u$ in $\Omega \times (0,\Tf)$
    and for $\alpha$ small enough, we keep track of the boundary values of $\tilde u$}.

    Because of the $\phi$-term, $\ye \in \Omega$ for $\e>0$ small enough. If $\se >0$, the
    $H$-inequality holds and we have $$\frac 1 \e + H\left (\ye,\se, p_\e -\frac{\alpha
    D_x\phi(\ye)}{[\phi(\ye)]^2}\right)\leq 0\; ,$$ where $\displaystyle p_\e = \frac{2(\ye-x)}{\e}=
    \frac{o(1)}{\e}$. Examining the $H$-term, it can be estimated by $$ \frac 1 \e -M(
    \frac{o(1)}{\e} +1) + \sup_{\alpha \in A} \left\{b(\ye,\se,\alpha)\cdot \frac{\alpha
    D_x\phi(\ye)}{[\phi(\ye)]^2}\right\}\leq 0\; ,$$ where $M$ takes into account the Lipschitz
    constant of $H(x,t,p_x)$ in $p_x$ (coming from boundedness of $b$) and the boundedness of $l$.

    By the assumption on $\phi$, the supremum is non-negative and therefore this inequality cannot
    hold for $\e$ small enough. This implies that necessarily, $\se=0$ and $\tilde u (\ye,\se) \leq
    u_0(\ye)$. Finally, letting $\alpha, \e \to 0$ with $\alpha\ll \e$, we obtain $\tilde u (x,0)
    \leq u_0(x)$.
 
    These inequalities at time $t=0$ being proved, we have just to apply the comparison result for
the stratified problem, Corollary~\ref{cor:comp.BACPID}, and the proof is complete.  \end{proof}

\begin{example}\label{ex:square} A standard example where Proposition~\ref{Dir-nonreg} can be
    applied is the square $[0,1] \times [0,1]$ in $\R^2$, with $$\varphi(x)=\varphi_i (x,t) \quad
    \hbox{on  }S_i,$$ where $S_1=]0,1[\times \{0\} $, $S_2=\{1\} \times ]0,1[ $ , $S_3=]0,1[\times
    \{1\}$ $S_4=\{0\} \times ]0,1[ $, each $\varphi_i$ being continuous on $S_i$. Of course, in
    order to have a function $\varphi$ which is W-adapted to the stratification, the values at the
    four corners are imposed by the values on each $S_i$ and obtained by computing their lower
    semi-continuous extensions. For example, at $(0,0)$ we have $\min (\varphi_1(0,t),
    \varphi_4(0,t))$.  We point out that $\varphi$ is still adapted if the values at the four
    corners are below these values. 

    If $H$ satisfies all the controllability conditions, then the first part
    Proposition~\ref{Dir-nonreg} applies.

    For the second one, the compatibility condition on $\domeg \times \{0\}$ should hold and for
    $\phi$, we can choose the distance to the boundary if $x$ is not located on one of the corners.
    In case of a corner, say $(0,0)$, we may choose, noting $x=(x_1,x_2)$, the function
    $\phi(x)=x_1x_2$, while for $(0,1)$, we may choose $\phi(x)=x_1(1-x_2)$, i.e. in each case the
product of the distances to the adjacent sides. The controllability condition ensures that the
requirement on $D_x\phi$ is satisfied.  \end{example}

\begin{remark}\label{rem:idp} We are not going to push very far the question of the existence of
    functions $\phi$ above playing the role of a distance function. Let us just mention that this
    should not be an issue in general, even if it might be difficult to provide a very general
    result.

    A convincing example is the case when $\overline \Omega$ is a convex set given by $$\overline
    \Omega := \bigcap_i \{x:\ p_i \cdot x \geq q_i\}\; ,$$ where the $p_i$ are in $\R^N$ and the
    $q_i$ in $\R$. The example of the square above can be generalized in the following way: if $x\in
    \domeg$ and if $I(x)$ is the set of indices $i$ for which $p_i \cdot x = q_i$, then one can
    choose $$ \phi(y):= \prod_{i\in I(x)}(p_i \cdot x - q_i)\; .$$ It is easy to check that the
    condition on $D_x\phi$ is satisfied as an easy consequence of the normal controllability since
    all the $p_i$ are clearly orthogonal to the space of $\Man{k}$ at $x$ is in $\Man{k}$.

    In the case of domains which are the complementaries of convex domains, namely $$\overline
\Omega := \bigcup_i \{x:\ p_i \cdot x \geq q_i\}\; ,$$ one can choose $$ \phi(y):= \sum_{i\in
I(x)}[(p_i \cdot x - q_i)_+]^2\; .$$ \end{remark}

\subsection{The case of non well-adapted data} 
\index{Dirichlet problems!non well-adapted boundary conditions}

We just described above some general framework for which the stratified formulation and the
classical viscosity solutions' one are in some sense equivalent. But let us also consider here the
case when the stratified formulation is unavoidable to get uniqueness. 

Let $\varphi$ be a \lsc function which is adapted, but not W-adapted to the stratification, \ie
assume that for some $(x,t)\in\Man{k}$, $$ \varphi(x,t)< \liminf_{\displaystyle
\mathop{\scriptstyle (y,s) \to (x,t)}_{\scriptstyle  (y,s) \in \Man{N}}} \varphi (y,s)\;.$$ Then,
there is no way that a subsolution---even after ``cleaning'' it---should satisfy $u \leq \varphi$ on
$\Man{k}$. This property has to be superimposed through the stratification formulation since the
Ishii one, using $\varphi^*$, will simply erase the small values of $\varphi$.

In this case we have the 
\begin{proposition}\label{Dir-nonreg-nW}\emph{--- Comparison for non well-adapted boundary
    conditions.}\smsp 
    Assume that \HSBC holds.  Let
$\varphi : \domeg \times [0,\Tf) \to \R$ be a lower-semicontinuous function, adapted to the
stratification.  \begin{enumerate} \item[$(i)$] If $u$ is an \usc viscosity subsolution of the
            Dirichlet problem such that \begin{equation}\label{con-Dir-Mk} u(x,t) \leq \varphi(x,t)
            \quad \hbox{for any  }(x,t) \in \Man{N-1}\cup\cdots\cup\Man{1}\; , \end{equation} then
            $\tilde u : \overline \Omega \times [0,\Tf) \to \R$ defined by $\tilde u (x,t)=u(x,t)$
            if $x\in \Omega$ and $$ \tilde u (x,t)= \limsup_{\displaystyle  \mathop{\scriptstyle
            (y,s) \to (x,t)}_{\scriptstyle y\in \Omega}} u(y,s)\quad \hbox{if  } x \in \domeg\; ,$$
        is a stratified subsolution of the problem.  \item[$(ii)$] If, in addition, we assume that
            \IDP holds and $u_0 \in C(\overline \Omega)$ satisfies $ u_0(x)\leq  \varphi_*(x,0)$ on
            $\domeg$, then for any viscosity supersolution of the Dirichlet problem, $$ \tilde u
            \leq v \quad \hbox{on  }\overline \Omega \times [0,\Tf) \; .$$ In particular, in this
            case there exists a unique continuous viscosity solution of the Dirichlet problem which
            satisfies \eqref{con-Dir-Mk}.  
\end{enumerate}
\end{proposition}

As we already explain it above, the key difference between Propositions~\ref{Dir-nonreg}
    and~\ref{Dir-nonreg-nW} is that the first one applies to all Ishii viscosity solutions while, in
    the second case, Condition~\ref{con-Dir-Mk} has to be imposed.

\begin{proof} We only give a sketch here since it follows the ideas of the proof of
    Proposition~\ref{Dir-nonreg}, namely\\[2mm] $(i)$ For any $k$, the condition on $\Man{k}$, \ie
    $$\max(u_t + H^k(x,t,D_x u), u - \varphi(x,t))\leq 0 \quad\hbox{on  }\Man{k}\; ,$$ where
    $$H^k(x,t,p_x)=\sup_{b(x,t,\alpha)\in T_{x} \tMan{k}}\{-b (x,t,\alpha) \cdot p_x
    -l(x,t,\alpha)\}\; ,$$ is obtained by combining \eqref{con-Dir-Mk} with an approximation ``from
    inside'', following Remark~\ref{MgMk}.\\[2mm] $(ii)$ The comparison result follows from the
    stratified formulation, while the existence is provided by the value function of the associated
    control problem.  \end{proof}

We conclude by an example showing the interest of the stratified formulation, related to
    Proposition~\ref{Dir-nonreg-nW}.

\begin{example}\label{ex:square2} We come back to an example in the square $S=[0,1] \times
    [0,1]\subset\R^2$.  The equation is $$ u_t +|D_x u|=1 \quad \hbox{in  }S\times (0,1)\;,$$ with
    the initial data $u(x,0)=0$ in $S$ and the (time-independent) Dirichlet boundary condition
    $$\varphi(x)=1 \quad \text{on  }\partial S\setminus \{0\}\;, \quad \varphi(0)=0\; .$$ Since
    $\varphi^*(x)\equiv 1$ on $\partial S$, it is easy to check that $u_1(x,t)=t$ is a classical
    viscosity solution of this problem for $0\leq t\leq 1$. Of course this first solution completely
    ignores the fact that $\varphi(0)=0$.

    On the other hand, $u_2(x,t)=\min(t,|x|)$ is also a solution of our problem but it satisfies
    $u_2(0,t)\leq 0$, \ie Condition \eqref{con-Dir-Mk}.  On this example, one can verify that
    Condition \eqref{con-Dir-Mk} is nothing but the main missing stratified inequality on $\Man{1}$,
    the other ones on $\Man{2}$ being also satisfied. We also point out that, at time $t=0$, it is
important to have the stratified inequality $u\leq \min(u_0,\varphi)$ on $\partial S$ to recover the
correct initial data, solving the $\F_{init}$ equation.  \end{example}

\section{On the Neumann problem}\label{sec:Neumann}
\index{Boundary conditions!Neumann}

In this section, we consider several cases where Neumann, or more generally oblique derivative
    boundary conditions arise, namely \begin{equation}\label{NP} \frac{\partial
    u}{\partial \gamma} =g(x,t) \quad \hbox{on $\domeg\times (0,\Tf)$\; ,} \end{equation} where
    $\gamma, g$ are bounded functions on $\domeg \times [0,\Tf]$, taking respectively values in
    $\R^N$ and $\R$. We recall that throughout this chapter, we make several simplifications
    referred to as \HSBC, listed on page~\pageref{page:simplifications}. In particular, we have a
    time-independent stratification of $\Omegb \times \R$ with
    $$\partial\Omega=\tMan{N-1}\cup\cdots\cup\tMan{0}\;,$$ and $\Man{k}= \tMan{k-1}\times \R$ for
    $k=1,..,N+1$.

For the assumptions on $\gamma$ and $g$, we anticipate the case of mixed boundary conditions and we
first introduce the following hypothese where $\omega$ is a connected, open subset of $\domeg$ which
is a $(N-1)$-dimensional $C^{1,1}$-submanifold of $\R^N$:\index{Neumann and oblique derivatives problems!natural assumptions}

\begin{assumption}{\Hgam[\omega]}{Natural Assumptions on $\gamma$ and $g$ on $\omega$.}\label{page:Hgam}\\[-1cm]
\begin{enumerate} \item[$(i)$] There exists $\nu>0$ and a Lipschitz continuous $\gamma_\omega:
            \R^N\times\R \to \R^N$ such that $\gamma=\gamma_\omega$ on $\omega\times[0,\Tf]$ and
            \begin{equation}\label{assump:NP} \gamma_\omega (x,t)\cdot n(x)\geq \nu >0\quad \hbox{on
            } \omega\times [0,\Tf], \end{equation} where $n(x)$ is the unit outward normal to $\domeg$
            at $x$\footnote{We point out that $\domeg$ and $\omega$ coincide in a neighborhood of
            each $x\in \omega$ and therefore $\domeg$ is smooth at such points as a consequence of
the assumptions on $\omega$.}.  \item[$(ii)$] There exists a continuous function $g_\omega:
\R^N\times \R\to \R$ such that $g=g_\omega$ on $\omega\times [0,\Tf]$.  \end{enumerate}
\end{assumption} We use \Hgam[\omega] for problems where we have an oblique derivative boundary
condition on $\omega\times [0,\Tf]$ and typically a Dirichlet boundary condition on the
complementary. For pure oblique derivative problem, we use the following assumption where we denote
by $(\tMan{N}_i)_{i\in I^N}$ the connected components of $\tMan{N}\subset \domeg$.

\begin{assumption}{\Hgamma}{Specific Hypotheses for the Oblique Derivative Problem.}\label{page:Hgams} For any $i\in
I^N$, \Hgam[\omega] holds for $\omega=\tMan{N}_i$ and we denote by $\gamma_i,g_i$ the corresponding
functions $\gamma_\omega,g_\omega$.  \end{assumption}

Several remarks can be made on these assumptions. First, notice that the above
$C^{1,1}$-assumption on $\omega$ is natural: as \Hgamma shows, we have in mind that $\omega$ is a
connected component of $\tMan{N-1}$, hence it should satisfy the classical regularity imposed on a
stratification. We recall that this $C^{1,1}$-regularity can be replaced by a $C^1$-one if
$H(x,t,p)$ is coercive in $p$, uniformly \wrt $x$ and $t$.

Next we point out that the assumptions on $\gamma$ are, of course, the same as those for
$b(x,t,\alpha)$ in \HBACP because, as it will become even more obvious later on, they play analogous
roles. Clearly, while the Lipschitz continuity in $x$ seems natural, the Lipschitz continuity in $t$
is quite restrictive. As for $b(x,t,\alpha)$, we refer the reader to Section~\ref{sect:mgdt} in
order to weaken this assumption.

Before coming back to $\gamma$ and $g$ and the exact sense of the notion of Ishii solution for the
oblique derivative problem, let us mention that the assumptions on $H$ are the same as in
Section~\ref{RefBF:Dir}: $H$ is given by \eqref{eq:Hmsh} with $b,l$ satisfying \HBACP and therefore
it satisfies \HBAHJ. We assume also that the normal controllability assumption \NCBCL holds, \ie
\eqref{eq:nca:dir} on $\Man{N}$ and more generally \eqref{NC:H2} on the various
manifolds~$(\Man{k})_k$.  

Obviously, we did not say anything on $\gamma$ and $g$ on $(\domeg \times (0,\Tf))\setminus
\Man{N}$. In fact, the notion of Ishii solution just uses their values on $\Man{N}$ in the following
way: if $(x,t) \in \domeg \times (0, \Tf)$, let us denote by $J(x,t)$ the set of $i$ such that
$(x,t) \in \overline{\Man{N}_i}$. Then the definition of Ishii sub and superslution is $$ \min
\Big(u_t+H(x,t,D_x u),\min_{i\in J(x,t)}(\gamma_i(x,t) \cdot D_x u-g_i(x,t))\Big)\leq 0\; ,$$ and $$
\max \Big(u_t+H(x,t,D_x u),\max_{i\in J(x,t)}(\gamma_i(x,t) \cdot D_x u-g_i(x,t))\Big)\leq 0\; .$$
In other words, only the values of $\gamma$ and $g$ on $\Man{N}$ really play a role. And an
analogous definition holds on $\domeg \times \{0\}$ with the additional $(u-u_0)$-term.

\

In this section, we look at the following cases.

\begin{enumerate} \item[1.] We first revisit the most classical case where both the boundary is smooth
            and the direction of reflection is Lipschitz continuous. Of course, here, the stratified
            approach does not bring any new result and this section just consists in describing the
            stratified formulation, which is rather different from the classical one.  \item[2.] We
                then consider the case of a smooth boundary with a codimension one discontinuity in
                the direction of reflection, a case which---to the best of our knowledge---is not so
                much investigated in the literature.  \item[3.] The two next cases can be called the
``Dupuis-Ishii'' configurations since they are those which these authors investigate in
\cite{DuIs2,DuIs1,DuIs-SDE}.  \end{enumerate}

Let us point out that, in \cite{DuIs2, DuIs1, DuIs-SDE}, Dupuis and Ishii study oblique derivative
problems in non-smooth domains for {\em possibly second-order, elliptic and parabolic, fully
nonlinear equations}, \ie in a far more general framework than ours.  They both prove comparison
results in two different cases that we describe below but they also obtain the uniqueness of
solutions for stochastic differential equations with oblique directions of reflection in domains
with corners. The two main cases that Dupuis and Ishii consider are the followings: 
\index{Neumann and oblique derivatives problems!Dupuis-Ishii configurations}

\medskip

\noindent\textbf{Configuration I} is the case of a smooth direction of reflection in domains which
satisfy only an exterior cone condition. More precisely, given $\gamma \in C^2(\R^N,\R^N)$ they
assume that there exists $\bar \delta,\eta>0$ such that, for any $0<\delta \leq \bar \delta$ and any
$x\in \domeg$, $$ B\big(x+\delta \gamma(x),\eta \delta\big)\subset \R^N\setminus\Omega\; .$$ Of
course, the $C^2$-regularity on $\gamma$ appears as a rather strong assumption but one has to keep
in mind that they obtain results for {\em second-order} equations. On the contrary, the assumption
on $\Omega$ is very weak, allowing corners and even worse configurations, \cf
Figure~\ref{fig:DupuisIshii}, \emph{left}. Concerning \Hgamma, this means that all the
$(\gamma_i)_{i\in I^N}$ can be taken equal and the same is true for the $(g_i)_{i\in I^N}$.

\medskip

\noindent\textbf{Configuration II} is the case when $\Omega$ is a bounded domain obtained as an
intersection: $$ \Omega=\bigcap_{i\in I} \Omega_i\; ,$$ where $I$ is a finite set of indices and the
$\Omega_i$ are $C^1$-domain, \cf Figure~\ref{fig:DupuisIshii}, \emph{right}. On the boundary of
each $\Omega_i$, the direction of reflection, denoted by $\gamma_i$, is assumed to be Lipschitz
continuous.  There are complicated assumptions which link $\gamma_i$ and $n_i$, the normal vector to
$\partial \Omega_i$ pointing outside $\Omega_i$.  We do not detail them here but let us just mention
that these conditions are inspired by those of Harrison and Reiman \cite{HaRe} and Varadhan and
Williams \cite{VaWi}, and they are known as being natural in this framework in order to obtain
comparison results (or uniqueness for stochastic differential equations with oblique directions of
reflection in such domains with corners). This case fully justifies the form of \Hgamma.

\begin{figure}[!h] 
    \begin{center}
    \includegraphics[width=0.9\textwidth]{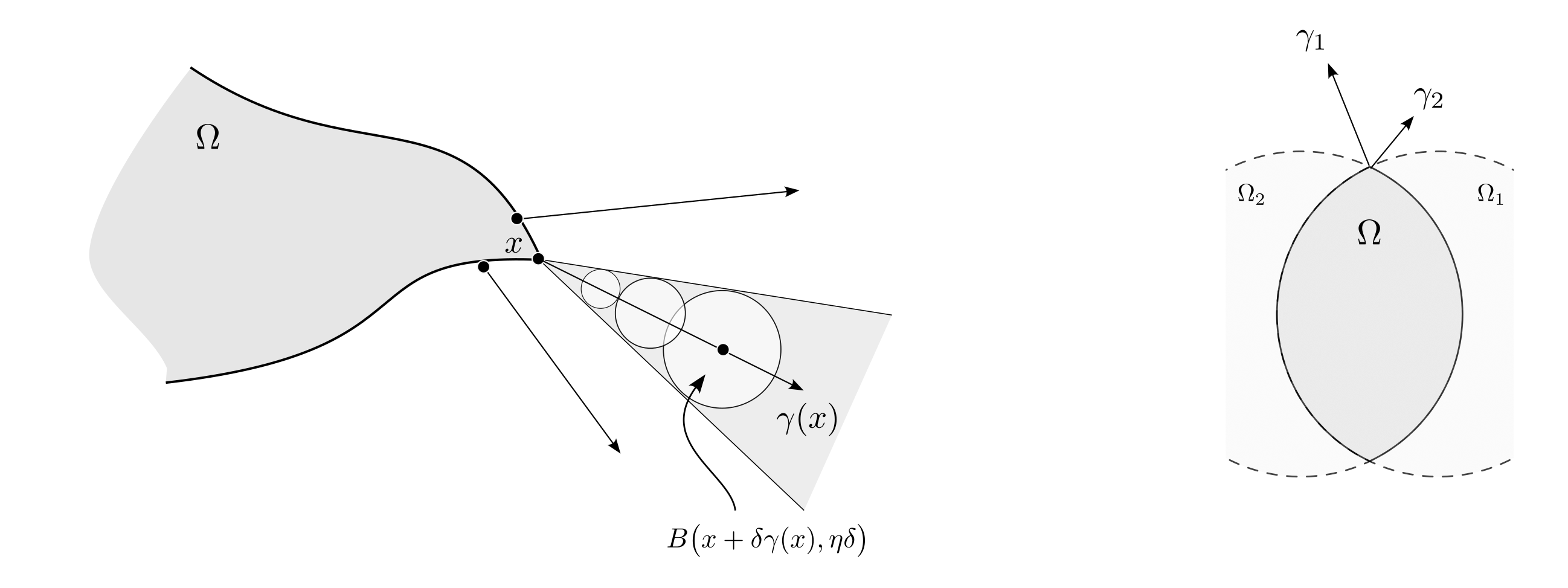}
    \caption{The Dupuis-Ishii configurations (\emph{left: }I, \emph{right: }II).}
    \label{fig:DupuisIshii}   
    \end{center}
\end{figure}

\bigskip

In the sequel, our aim is to treat these two configurations with, of course, some restrictive
assumptions due to the stratified approach.

More generally, in the four frameworks we have mentioned above, our aim is to give conditions under
which Ishii's (sub)solutions are stratified (sub)solutions. Of course, since stratified
supersolutions are just Ishii supersolutions, only the case of subsolutions has to be considered.
Sometimes we give full results, sometimes we just give indications on how to address the problem if
it is too complicated to state a general result. We recall that for simplicity, $\Omega$ is bounded
here but under suitable modifications, similar results are valid in the unbounded case too.

We conclude this introduction by showing that time $t=0$ does not cause any problem under natural
assumptions: this is a consequence of the following result whose proof is based on arguments of
Proposition~\ref{reg-sub}. Actually these arguments, together with those relying on the
\NCBCL-assumption on the Hamiltonian, also allow to prove that subsolutions are regular on the
boundary for $t>0$.  Under the simplifications we make, the initial stratification is nothing but
$\M_0=(\tMan{k})_{k=0..(N-1)}$, but the result holds in general, even if $\M_0$ is not the trace of
$\M$ at $t=0$.

\begin{proposition}\label{init-Neumann} Assume that $\Omegb$ is a stratified domain associated to a
    stratification $\M_0$ and consider the problem \begin{equation}\label{pb:neumann.initial}
        \begin{cases} u=u_0(x) & \hbox{in  }\Omega\;,\\ \min\big(u-u_0(x), \G(x,D_x u)\big)=0 &
        \hbox{on  }\domeg\;.  \end{cases} \end{equation}

    \noindent $(i)$ Assume that, for any $\xb \in \Man{k}_0 \cap \domeg$, there exists $e \in
        (T_{\xb} \Man{k}_0)^\bot$ and $\nu,K,r>0$ such that, for any $x\in \Man{k}_0 \cap \domeg
        \cap B(\xb,r)$ $$ \G(x,p_x + Ce) \geq \nu C -K(|p_x|+1)\;.$$ Then, any  \usc viscosity
        subsolution $u$ of \eqref{pb:neumann.initial} satisfies $u\leq u_0(x)$ on $\domeg$.

    \smallskip

    \noindent$(ii)$ Assume that, for any $\xb \in \Man{k}_0 \cap \domeg$, there exists $e \in
        (T_{\xb} \Man{k}_0)^\bot$ and $\nu,K,r>0$ such that, for any $x\in \Man{k}_0 \cap \domeg
        \cap B(\xb,r)$ $$\G(x,p_x + Ce) \leq - \nu C + K(|p_x|+1)\;.$$ Then, any \lsc viscosity
        supersolution $v$ of \eqref{pb:neumann.initial} satisfies $v\geq u_0(x)$ on $\domeg$.
    \end{proposition} We leave the easy proof of this result to the reader since, as we mention it
    above, it is based on the arguments of the proof of Proposition~\ref{reg-sub}. 

Let us point out that 
\begin{enumerate} \item[$(i)$] thanks to Proposition~\ref{visc-ineq-init}, or at
            least by borrowing the arguments in its proof, a sub or supersolution of an oblique
            derivative boundary condition $$\frac{\partial u}{\partial \gamma} =g(x,t) \quad
            \hbox{on }\domeg\times (0,\Tf)$$ when $\gamma$ and $g$ are continuous,  typically
            satisfies the conditions of Proposition~\ref{init-Neumann} with
            $$\G(x,p_x)=\gamma(x,0)\cdot p_x - g(x,0)\;.$$

    \item[$(ii)$] In the case of several directions of reflection nearby $\Man{k}$, \ie when
        $\gamma$ and $g$ are discontinuous, existence of a vector~$e$ as above is a natural
        assumption on $\gamma$ (or the various $\gamma_i$ involved) provided \eqref{assump:NP}
        holds, for example. In this case we apply Proposition~\ref{init-Neumann} by considering
        different Hamiltonians for the subsolution and the supersolution, introducing respectively
        $$ \G(x,p_x)=\min_i(\gamma_i(x,0)\cdot p_x - g_i(x,0))\; , \; \tilde
        \G(x,p_x)=\max_i(\gamma_i(x,0)\cdot p_x - g_i(x,0))\;.$$ Using $\G$ for the subsolution and
$\tilde\G$ for the supersolution leads to the desired result, $u(x,0)\leq u_0(x)\leq v(x,0)$ on
$\domeg$.  
\end{enumerate}

    Of course, similar remarks hold for nonlinear boundary conditions of Neumann type.  For this reason,
    we will always assume in the sequel that \HBAIDCP holds since Proposition~\ref{init-Neumann}
    gives this property in an easy and natural way and we concentrate on the stratified formulation
    for $t>0$.

\subsection{Stratified formulation of the classical case} \label{subsec:classical.oblique}
\index{Neumann and oblique derivatives problems!stratified formulation}

As for the Dirichlet problem, we begin with the most standard framework: an oblique derivative
    problem in a smooth domain. More precisely, we consider the standard problem introduced in
    \eqref{standardHJB}, namely \begin{equation} \begin{cases}\label{hjb.standard.neumann} u_t +
        H(x,t,D_x u)=0 & \hbox{in $ \Omega \times (0,\Tf)$}\;, \\ u(x,0) =u_0 (x) & \hbox{in
    $\Omega$}\;, \end{cases} \end{equation} associated with the boundary condition \eqref{NP}.
    Because of these hypotheses, the situation reduces to $\Man{N+1}=\Omega \times (0,\Tf)$ and
    $\Man{N}=\domeg\times (0,\Tf)$.

The first key difference with the Dirichlet problem is that viscosity subsolutions are regular at
the boundary and therefore we do not need to redefine them on the boundary. More precisely
    \begin{proposition}\emph{--- Regularity of subsolutions..}\smsp
        Let $\Omega$ be a bounded $C^{1,1}$-smooth domain. Assume that \HSBC and
    \Hgamma hold\footnote{Here, of course, \Hgamma is the same as \Hgam[\domeg].}. Then, any \usc
subsolution of \eqref{hjb.standard.neumann}-\eqref{NP} is regular at the boundary for $t>0$.
\end{proposition}

\begin{proof} Let $u$ be an \usc subsolution of \eqref{hjb.standard.neumann}-\eqref{NP} and $(x,t)
    \in \domeg \times (0,\Tf)$. If $u$ is not regular at $(x,t)$ this means that $$
    u(x,t)>\limsup_{\displaystyle  \mathop{\scriptstyle (y,s) \to (x,t)}_{\scriptstyle  (y,s) \in
    \Man{N+1}}} u (y,s)\; .$$ We consider, for $0<\e \ll 1$, the function defined on $\Man{N}$ by
    $$(y,s)\mapsto u(y,s)-\frac{(s-t)^2}{\e^2}-\frac{|y-x|^2}{\e^2}\; .$$ This function has a local
    maximum point at $(\ye,\se)$ near $(x,t)$ and $u(\ye,\se)\to u(x,t)$ as $\e \to 0$. But the jump
    of $u$ on the boundary implies that necessarily $(\ye,\se)\in \domeg \times (0,\Tf)$. 

    Notice that the distance function to the boundary $\domeg$, denoted by $d(\cdot)$, is $C^1$ in a
    neighborhood of $\domeg$ by the assumption on the regularity of $\Omega$. Hence, given any
    $\lambda\in\R$, $(\ye,\se)$ is also a local maximum point of the function $\Psi_\lambda$ defined
    on $\overline \Omega \times (0,\Tf)$ by
    $$\Psi_\lambda(y,s):=u(y,s)-\frac{(s-t)^2}{\e^2}-\frac{|y-x|^2}{\e^2}-\lambda d(y)\;.$$

    Using the Ishii viscosity inequality on the boundary implies that $$ \min\Big(a_\e+
    H\big(\ye,\se, \pe-\lambda n(\ye)\big)\;,\,\big(\pe-\lambda n(\ye)\big)\cdot
    \gamma(\ye)-g(\ye,\se)\Big)\leq 0\;,$$ where $$ a_\e:= \frac{2(\se-t)}{\e^2}\quad
    \hbox{and}\quad \pe:=\frac{2(\ye-x)}{\e^2}\; .$$ But of course, for $\lambda<0$ large enough, we
    obtain a contradiction because of the normal controllability assumption on $H$ and the
assumption on $\gamma$, which ends the proof.  \end{proof}

In this simple case, it remains to identify the $\F^{N}$-inequality on $\Man{N}$ and to show the
equivalence between Ishii viscosity (sub)solutions and stratified (sub)solutions. As we did for the
Dirichlet case, we enlarge the set $\BCL$ on the boundary to take into account the boundary
condition. Here, the enlargement consists in adding triplets of the form
$((-\gamma(x,t),0),0,g(x,t))$, assigning the cost $g(x,t)$ to a reflection-type boundary dynamic
$-\gamma(x,t)$ on $\partial\Omega$.

The result is the following 
\begin{proposition}\label{OD-MN}
    Let $\Omega$ be a bounded
    $C^{1,1}$-smooth domain. Assume that \HSBC and \Hgamma hold.  If $u$ is a viscosity
    subsolution of the oblique derivative problem, it is a stratified subsolution of the problem
    with $$\F^N(x,t,(p_x,p_t))= \sup \Big\{\theta p_t -\big(\theta \gb^x- (1-\theta)\gamma\big)\cdot
    p_x -\big(\theta \gl +(1-\theta)g\big)\Big\}\;\hbox{on }\Man{N}\; ,$$ where the supremum is
    taken on all $(\gb,0,\gl) \in \BCL (x,t)$ such that there exists $\theta \in [0,1]$ satisfying
    $(\theta \gb^x-(1-\theta)\gamma)\cdot n(x)=0$, where $n(x)$ is the unit outward normal to
$\domeg$ at $x$.  
\end{proposition}

Notice that in Proposition~\ref{OD-MN}, we have used lighter notations but it is clear that
$\gb^x=b(x,t,\alpha)$ for some $\alpha \in A$, $\gl=l(x,t,\alpha)$ and $\gb=(\gb^x,-1)$.  We also
point out that the parameters $\theta$ which satisfy $(\theta \gb^x-(1-\theta)\gamma)\cdot n(x)=0$
for some $x,t,\gb^x$ are bounded away from $0$ because of \Hgamma: indeed 
$$ \theta (\gb^x+\gamma)\cdot n(x)\geq \nu >0\; .$$  
This property implies that the $\F^N$-Hamiltonians are strictly increasing in $p_t$, uniformly \wrt
$x,t,p_x$ and they can be expressed as a ``$u_t+H^N(x,t,D_xu)$''-one since we may divide by $\theta$
inside the $\sup$. This remark which is also true for the various $(\F^k)$ allows a simple checking
of \LOCb. 

We also take this opportunity to recall that, thanks to Lemma~\ref{tgfields}, for $1\leq k\leq N$,
each $\F^k$ satisfies the needed ``good assumptions''. In particular, it can be shown that \HBAHJ
holds for $H^N(x,t,p_x)$, a not completely obvious fact.

In the case of unbounded domains, analogous properties play a key role, in particular for the
checking of Assumption~\LOCa : the presence of a $u_t$-term allows to better localize the equation,
see Chapter~\ref{chap:openpb-partV}. But in order to be true, such properties require suitable
assumptions on the boundary and the direction of reflection.

\begin{proof}
    We have to show that, if $\phi$ is a smooth function and if $(x,t)\in \Man{N}$ is a strict local
    maximum point of $u-\phi$ then 
    $$\theta \phi_t (x,t) -(\theta \gb^x-(1-\theta)\gamma)\cdot D_x\phi(x,t)
    -(\theta \gl +(1-\theta)g) \leq 0 \;,$$
    for any $\gb,\gl,\theta$ satisfying the conditions of Proposition~\ref{OD-MN}.

    \smallskip

    \noindent\textbf{(a)} 
    To do so, we introduce $\lambda \in \R$, defined as the unique solution of the equation
    \begin{equation}\label{eq:def.lambda.neuman}
    \gamma (x,t)\cdot (D_x\phi(x,t)-\lambda n(x))=g(x,t)\;.
    \end{equation}
    Notice that since $\gamma(x,t)\cdot n(x)\geq\nu>0$, $\lambda$ is well-defined.
    Then, we consider the function 
    $$ \Psi_\e:(y,s) \mapsto u(y,s)-\phi(y,s) -(\lambda-\delta)d(y) -\frac{[d(y)]^2}{\e^2}\; ,$$
    for $0<\e,\delta\ll 1$. We recall that, as above, $d$ denotes the distance function to the
    boundary $\domeg$ and that $D_xd(x)=-n(x)$ on $\domeg$; we will use the notation
    $n(x)$ for $-D_xd(x)$ even if $x$ is not on the boundary.

    \smallskip

    \noindent\textbf{(b)}
    We first fix $\delta>0$. If $\e$ is small enough, $\Psi_\e$ has a local maximum point at
    $(\xe,\te)$ and $(\xe,\te)\to (x,t)$ as $\e \to 0$ by the maximum point property of $(x,t)$. 

    If $(\xe,\te)\in \domeg \times (0,\Tf)$, we claim that the $\gamma$-inequality cannot hold.
    Indeed otherwise 
    $$ \gamma(\xe,\te)\cdot \big(D_x \phi (\xe,\te)-(\lambda-\delta)n(\xe)\big)\leq
    g(\xe,\te)\;,$$
    which cannot be valid for $\e$ small enough because of the definition of $\lambda$ and the fact
    that $\delta>0$ and $\gamma (\xe,\te)\cdot n(\xe)\geq \nu >0$. Hence, necessarily the
    $H$-inequality holds at $(\xe,\te)$, as well as for interior points.

    \smallskip

    \noindent\textbf{(c)}
    For any $(\gb,0,\gl) \in \BCL(x,t)$, there exists a control $\alpha$ such that $\gb^x=b(x,t,\alpha)$
    and $\gl=l(x,t,\alpha)$. Choosing $(\gb_\e,0,\gl_\e)\in\BCL(\xe,\te)$ such that 
    $$(\gb_\e,0,\gl_\e)=\big((b(\xe,\te,\alpha),-1),0,l(\xe,\te,\alpha)\big)\;,$$
    we have, as a particular case of the $H$-inequality,
    $$\phi_t (\xe,\te) -\gb_\e^x \cdot \Big(D_x\phi(\xe,\te)-
    (\lambda-\delta)n(\xe)-\frac{2d(\xe)}{\e^2}n(\xe)\Big) - \gl_\e  \leq 0\; .$$
    Taking $(\gb,0,\gl) \in \BCL(x,t)$ and $\theta$ such that the property $(\theta
    \gb^x-(1-\theta)\gamma)\cdot n(x)=0$ holds, we first deduce that $\theta >0$ since
    $\gamma(x,t)\cdot n(x)>0$, and then that $\gb^x\cdot n(x)>0$. 

    Therefore, for $\e$ small enough, $\gb_\e^x \cdot n(\xe)>0$ and we can drop the
    $2d(\xe)(\gb_\e^x \cdot n(\xe))\e^{-2}$ term in the above inequality.
    Letting $\e \to 0$, yields
    $$\phi_t (x,t) -\gb^x \cdot \Big(D_x\phi(x,t)-(\lambda-\delta)n(x)\Big) - \gl  \leq 0\; .$$
    Finally, we let $\delta \to 0$ and the conclusion follows by using the $\theta$-convex
    combination of this inequality with \eqref{eq:def.lambda.neuman}.
\end{proof}

Several remarks after this result. 
\begin{enumerate}
\item[$(i)$] It is clear enough from the proof that the case of {\em sliding boundary conditions},
    \ie $$u_t + \frac{\partial u}{\partial \gamma} =g(x,t) \quad \hbox{on }\domeg\times (0,\Tf)\;,$$
    can be treated exactly in the same way. 
\item[$(ii)$] Less obviously (but this is still easy), the case where there is a control on the
    reflection
    $$\sup_\beta \left\{ \gamma_\beta \cdot D_x u -g_\beta\right\}=0 \quad \hbox{on }\domeg\times (0,\Tf)\;,$$
        where the set of $(\gamma_\beta,g_\beta)$ is convex and
        continuous in $(x,t)$, can also be treated\footnote{In general this set is not convex but we
        can take a convex enveloppe and this does not change the ``sup'' in the boundary
        condition.}. It is easy to check that one has just to repeat the above arguments for $(b,l)$
        and $(\gamma_\beta,g_\beta)$ such that $(\theta \gb^x-(1-\theta)\gamma_\beta)\cdot n(x)=0$.
\item[$(iii)$] But, as it may be expected, the stratified formulation does not bring new results
    as long as all data are continuous.

\item[$(iv)$]
    In the case of a Neumann boundary condition of the form
    $$ 
    \frac{\partial u}{\partial n} =g(x,t) \quad \hbox{on }\domeg\times (0,\Tf)\; ,$$
    the supremum in $\F^N$ is taken on all $(\gb,0,\gl) \in \BCL (x,t)$ such that there exists $\theta
        \in [0,1]$ such that $(\theta \gb^x -(1-\theta)n(x))\cdot n(x)=0$, which reduces to
    $$\theta \gb^x\cdot n(x) -(1-\theta)=0\; .$$
    Therefore, $\gb^x \cdot n(x) \geq 0$ and $\theta= (1+\gb^x\cdot n(x))^{-1}$. Decomposing $\gb^x=
    \gb^{x,\bot}+\gb^{x,\top}$ where $\gb^{x,\bot}$ is the projection of $\gb^x$ on the normal direction and
    $\gb^{x,\top}$ on the tangent space of $\domeg$ at $x$, we have to look at the supremum of
        $$\theta \Big(p_t - \gb^{x,\top} \cdot p_x - \big(\gl+\gb^x\cdot n(x) g(x,t)\big)\Big)$$
    for $\gb^x \cdot n(x) \geq 0$, since $(1-\theta)=\theta \gb^x \cdot n(x)$. 

    As we already mentioned, $\theta$ cannot vanish and the condition reduces to
    $$u_t + \sup_{\gb^x \cdot n(x) \geq 0}\Big(\gb^{x,\top} \cdot D_xu -
        \big(\gl+\gb^x\cdot n(x) g(x,t)\big)\Big)\leq 0\;.$$

    Now, when looking at the reflected trajectory for a control problem, one has to solve an ode like
    $$ \dot X (s) = \gb^x(s) -\1_{\{X(s) \in \domeg\}}n(X(s)).d|k|_s \; ,$$
    where $|k|_s$ is the process with bounded variation which keeps the trajectory inside $\Omegb$
    and the associated cost is 
    $$ \int_0^t \gl(s)ds + \int_0^t g(X(s),s)\1_{\{X(s) \in \domeg\}}d|k|_s \; .$$
    It is easy to see that $d|k|_s= \1_{\{X(s) \in \domeg\}} \gb^x(s).n(X(s))ds$ if $\gb^x(s)\cdot
    n(X(s))\geq 0$ and the cost becomes 
    $$ \int_0^t (\gl(s) +\1_{\{X(s) \in \domeg\}} \gb^x(s).n(X(s)))ds \;,$$
    which is exactly what the stratified formulation is seeing on the boundary.
\end{enumerate}

\subsection{Codimension one discontinuities in the direction of reflection}
\label{subsec:codim.one.neumann}
\index{Neumann and oblique derivatives problems!codimension one discontinuities}

In this section we consider the case of discontinuous directions of reflection $\gamma(\cdot)$, but
restrict ourselves to the codimension one case. For starters, we first address the case of a
``flat'' discontinuity in $\R^2$, and then generalize the result to $\R^N$.

\bigskip

\noindent\textbf{Flat discontinuity on a line in $\R^2$ ---} The situation is depicted in
Fig.~\ref{fig:neuman.flat} below: we assume that $\gamma_1$ and $\gamma_2$ are constant directions
of reflexion, outward pointing and satisfying $\det(\gamma_1,\gamma_2)<0$. We come back on the cases
$\det(\gamma_1,\gamma_2)=0$ or $>0$ just after Proposition~\ref{ODFlat-M1}. The hypotheses on $H$
are the same as in Section~\ref{subsec:classical.oblique}.

Of course here, we just need to study the situation ocurring at $\Man{1}=\{(0,0)\}\times (0,\Tf)$
since the arguments of Proposition~\ref{OD-MN} prove that Ishii subsolutions satisfy the stratified
conditions on $\Man{2}=[(-\infty,0) \times \{0\}\cup (0,+\infty) \times \{0\}] \times (0,\Tf)$.

\begin{figure}[!h]
    \begin{center}
    \includegraphics[width=0.5\textwidth]{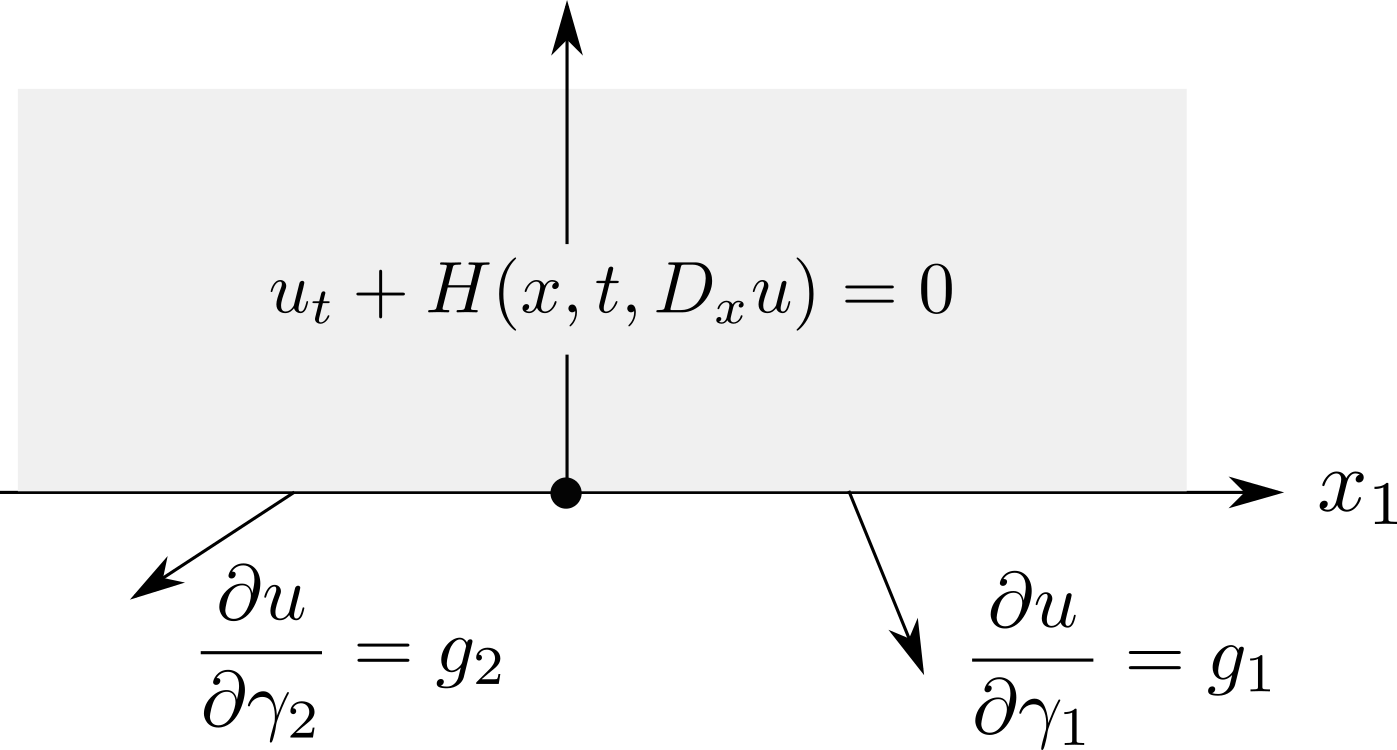}
    \caption{Flat discontinuous oblique derivative problem}
    \label{fig:neuman.flat}    
    \end{center}
\end{figure}

\begin{proposition}\label{ODFlat-M1}
    Assume that \HSBC and \Hgamma hold with constant directions of
    reflections $\gamma_1,\gamma_2$ satisfying moreover $\det(\gamma_1,\gamma_2)<0$. 

    If $u$ is an \usc viscosity subsolution of the above oblique
    derivative problem, it is a stratified subsolution of the problem associated with
    $$\F^1((p_x,p_t))= \sup
    \Big\{\theta_3 p_t -\big(\theta_3 \gl +\theta_1 g_1+\theta_2 g_2\big)\Big\}\;\hbox{on }\Man{1}\; ,$$
    where the supremum is taken on all $(\gb,0,\gl) \in \BCL (0,t)$ such that there exists
    $\theta_1,\theta_2,\theta_3 \in [0,1]$ satisfying $\theta_1+\theta_2+\theta_3=1$ and
    $\theta_1\gamma_1 + \theta_2\gamma_2 -\theta_3 \gb^x=0$.  
\end{proposition}

\begin{proof}
    As we said, we only focus on $\Man{1}=\{(0,0)\}\times(0,\Tf)$.
    Let $\phi$ be a $C^1$-function on $\R$ and $t$ be a strict local maximum
    point of the function $s \mapsto u(0,s)-\phi(s)$. We have to show that, if
    $\theta_1,\theta_2,\theta_3, \gamma_1 ,\gamma_2 , \gb$ satisfy the property
    which is required in Proposition~\ref{ODFlat-M1}, then $$\theta_3
    \phi'(t) -(\theta_3 \gl -\theta_1 g_1-\theta_2 g_2)\leq 0\; .$$ It is worth
    pointing out that we can do that only if $(\gb,0,\gl)$ is in the interior of
    $\BCL(0,t)$, a point that we will use in the proof.

    \medskip

    \noindent\textbf{(a)} Let us fix $\delta>0$ small and let us build $p_\delta\in\R^2$ such that
    \begin{equation}\label{def:pdelta1} 
        p_\delta\cdot \gamma_1=g_1+\delta\quad , \quad p_\delta\cdot \gamma_2=g_2+\delta\;,
    \end{equation}
    noticing that such a $p_\delta$ exists because of the assumptions on $\gamma_1,\gamma_2$.
    Next, we introduce the function
    $$ (y,s) \mapsto u(y,s)-\phi(s) -p_\delta\cdot y -\frac{Ay\cdot y}{\e^2}\; .$$
    where $A$ is a symmetric, positive definite matrix. Additional properties on $A$ will be
    needed and described all along the proof and at the end, we will show that such a matrix exists.

    This function has a maximum point at $(\xe,\te)$ and $(\xe,\te)\to (0,t)$ as $\e \to 0$ by the
    strict local maximum point property of $(0,t)$. Now we examine the different possibilities: if
    $\xe =((\xe)_1,0)= (\xe)_1 e_1$ with $(\xe)_1\geq 0$ and $e_1=(1,0)$, we get
    $$\Big(p_\delta+\frac{2A\xe}{\e^2}\Big)\cdot \gamma_1 -g_1 = \delta + \frac{2\xe}{\e^2}\cdot A\gamma_1
    = \delta + \frac{2(\xe)_1e_1}{\e^2}\cdot A\gamma_1 \;.$$
    Hence, if $A\gamma_1\cdot e_1 \geq 0$, since $\delta>0$ the inequality 
    ``$\displaystyle \frac{\partial u}{\partial \gamma_1}\leq g_1$'' cannot hold. 
    Similarly, if $(\xe)_1\leq 0$, the inequality 
    ``$\displaystyle \frac{\partial u}{\partial\gamma_2}\leq g_2$'' cannot hold provided
    $A\gamma_2\cdot e_1 \leq 0$.

    Therefore, wherever $\xe$ is, the $H$-inequality always holds and, since $(\gb,0,\gl)$ is in the
    interior of $\BCL(0,t)$, for $\e$ small enough, $(\gb,0,\gl) \in \BCL(\xe,\te)$ which implies
    $$  \phi'(\te) -\gb^x \cdot\Big(p_\delta+\frac{2A\xe}{\e^2}\Big)\leq \gl \; .$$

    \noindent\textbf{(b)} Now we examine the $\gb^x$-term, remarking that $\theta_3$ cannot be $0$ and
    using that $A$ is symmetric:
    \begin{align}
        -\gb^x \cdot \frac{2A\xe}{\e^2}= & -\frac1{\theta_3}(\theta_1\gamma_1+\theta_2\gamma_2)\cdot 
        \frac{2A\xe}{\e^2}\\
         = & - \frac1{\theta_3}(\theta_1A\gamma_1+\theta_2 A\gamma_2)\cdot  \frac{2\xe}{\e^2}\;.
     \end{align}
    A natural constraint on $A$ is $A\Gamma = -e_2$ where $\Gamma:=\theta_1\gamma_1+\theta_2 \gamma_2$
    since in this case 
    $$ (\theta_1A\gamma_1+\theta_2 A\gamma_2)\cdot
     \frac{2\xe}{\e^2}= A\Gamma\cdot \frac{2\xe}{\e^2}= -e_2 \cdot \frac{2\xe}{\e^2}=
     -\frac{2(\xe)_2}{\e^2}\leq 0\; ,$$
    and therefore the term $\displaystyle -\gb^x \cdot \frac{2A\xe}{\e^2}$ is nonnegative.
 
    In this case, taking the convex combination of the inequality $\phi'(\te) -\gb^x \cdot p_\delta
    \leq \gl $ with those of \eqref{def:pdelta1} gives the result.

    \medskip

    \noindent\textbf{(c)} In order to conclude the proof, we have to investigate the existence of
    a matrix $A$ satisfying the above requirements. To do so, we introduce first a symmetric matrix
    $A^{-1}$ under the form 
    $$A^{-1}=\left(\begin{array}{cc}\alpha & -\Gamma_1 \\ -\Gamma_1 & -\Gamma_2\end{array}\right)$$
    for some suitable parameter $\alpha >0$, its second column being imposed by the property
    $A\Gamma=-e_2$. Notice that 
    $$ -\Gamma_2=-\Gamma \cdot e_2=-(\theta_1\gamma_1+\theta_2 \gamma_2)\cdot e_2 >0\; ,$$
    by the properties of $\gamma_1, \gamma_2$. Therefore if $\alpha>0$ is large enough,
    $\det(A^{-1})>0$ which implies that $A^{-1}$ is symmetric, positive definite.
    We deduce that $A$ has the form
    $$A= \frac{1}{\det(A^{-1})}\left(\begin{array}{cc}-\Gamma_2 & \Gamma_1 \\ \Gamma_1 & \alpha
    \end{array}\right)\;.$$
    Now, notice that 
    $$\begin{aligned}
        A\gamma_1\cdot e_1 = \gamma_1 \cdot A e_1 &=
    \frac{1}{\det(A^{-1})}\big(-\gamma_{1,1}\Gamma_2+\gamma_{1,2}\Gamma_1\big)\\
        &= \frac{1}{\det(A^{-1})}\det(\Gamma,\gamma_1)\\[2mm]
        &= \theta_2\det(A)\det(\gamma_2,\gamma_1)\;.
    \end{aligned}$$
    So, since $\det(A)>0$, the sign of $A\gamma_1\cdot e_1$ is the same as the sign of
    $\det(\gamma_2,\gamma_1)$.  Similarly, the sign of $A\gamma_2\cdot e_1 = \gamma_2 \cdot A e_1$
    is the same as the one of $\theta_1\det(\gamma_1,\gamma_2)$.

    Therefore the matrix $A$ we look for exists (taking $\alpha >0$ large enough) provided
    $\det(\gamma_1,\gamma_2)\leq 0$, which completes the proof.
\end{proof}

Let us comment on the complementary situation $\det(\gamma_1,\gamma_2)\geq0$ :

\noindent $(i)$ in the case $\det(\gamma_1,\gamma_2)=0$, we can assume \wlg that
            $\gamma_1=\gamma_2$ (typically: $\gamma_1=\gamma_2= -e_2$). This situation can be
            treated by the methods of the proof of Proposition~\ref{prop:DuIs1} below. Here the
            discontinuity is just in the cost $g_1,g_2$ and it is clear that, if $g_1<g_2$, $\F^1$
            just reduces to 
            $$\F^1((p_x,p_t))= \sup \left\{\theta_3 p_t -(\theta_3 \gl +\theta_1 g_1)
            \right\}\;\hbox{on }\Man{1}\; ,$$
            where the supremum is taken on all $(\gb,0,\gl) \in \BCL (0,t)$ such that there exists
            $\theta_1,\theta_3 \in [0,1]$ such that $\theta_1+\theta_3=1$ and $\theta_1\gamma_1
            -\theta_3 \gb^x=0$. Indeed, the supremum is clearly achieved for the lowest cost $g_1$ and
            the proof of Proposition~\ref{prop:DuIs1} shows how to deduce the $\F^1$-inequality from
            the $\F^2$-one on $[(0,+\infty)\times \{0\}]\times (0,\Tf)$.

\noindent $(ii)$ If we assume that $\gamma_1\cdot e_2=\gamma_2\cdot e_2=-1$, the condition
            $\det(\gamma_1,\gamma_2)\leq0$ reduces to the tangential components inequality
            $\gamma_{2,1}\leq\gamma_{1,1}$. Fig.~\ref{fig:neuman.flat.config} below shows different
            types of situation where $\gamma_1,\gamma_2$ and their oppposite---in dashed lines---are
            shown, those opposites being involved in the reflexion process that occurs on the
            boundary.

            From the first two examples it could be guessed that the trajectories in the good case
            are more of a ``regular'' type. However, the other examples show that we can also allow
            some cases where the reflexions go in the same direction, provided some ``squeezing''
            effect holds.  

\begin{figure}[!h]
    \begin{center}
    \includegraphics[width=0.75\textwidth]{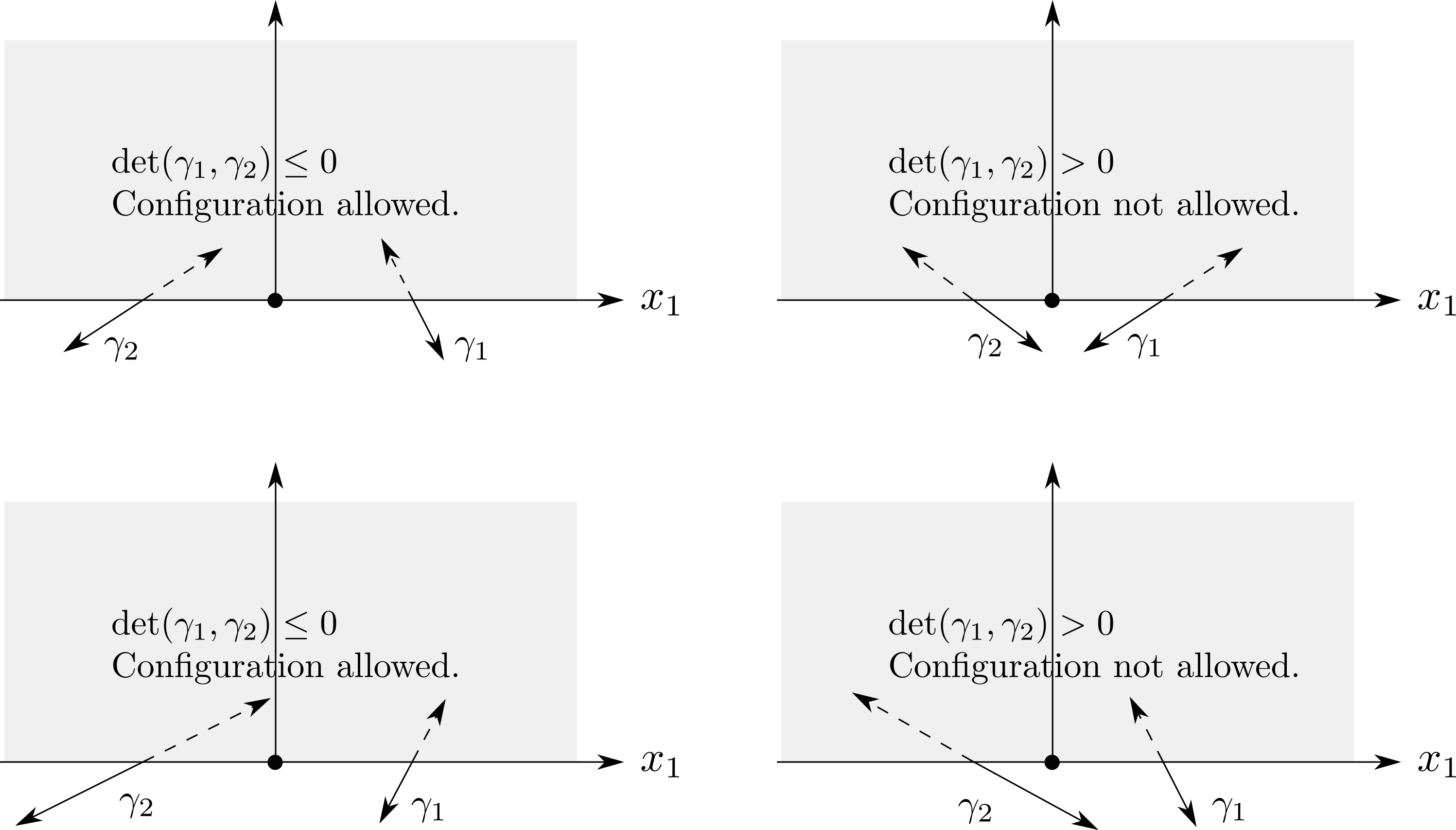}
    \caption{Configurations for a discontinuous oblique derivative problem}
    \label{fig:neuman.flat.config}    
    \end{center}
\end{figure}

\

\bigskip

\noindent\textbf{General codimension one discontinuities in $\R^N$ ---} We consider now the
case of a general stratified domain $\Omegb\times[0,\Tf)$, \cf Definition~\ref{def:stratdom} but we
make some specific assumptions on the structure of the stratification:
\begin{enumerate}
    \item[$(i)$] the set $\Omega$ is a bounded, $C^{1,1}$-smooth domain;
    \item[$(ii)$] we decompose $\domeg$ as $\tMan{N-1}_1\cup\H\cup\tMan{N-1}_2$ where
        $\H=\tMan{N-2}$;
    \item[$(iii)$] the stratification of $\Omegb\times(0,\Tf)$ is given by
        $\tMan{N-1}_{1,2}\times(0,\Tf)$ and $\H\times(0,\Tf)$.
\end{enumerate}
The notation $\H$ for $\tMan{N-2}$ is for the fact that $\H$ plays the role of
an hyperplane in $\domeg$, separating the two components $\Man{N-1}_{1,2}$ of $\domeg$. 

On $(\tMan{N-1}_{1}\cup\tMan{N-1}_2)\times(0,\Tf)$, we get the oblique derivative boundary condition
$$ \frac{\partial u}{\partial \gamma_i}=g_i(x,t) \quad \hbox{on }\tMan{N-1}_i\times(0,\Tf) \; ,$$
where we assume that the directions of reflections $\gamma_i$ depend only on $x$, not on $t$, and
are Lipschitz continuous function in $\R^N$, while the costs $g_i$ are continuous. Here again, the
hypotheses on $H$ are the same as in Section~\ref{subsec:classical.oblique}.

In this framework, we have the
\begin{proposition}\label{OD-gamma-disc}\emph{--- Comparison for discontinuous Neumann conditions.}\smsp
    Assume that \HSBC and \Hgamma hold.
    Assume moreover that
    \begin{enumerate}
        \item[$(i)$] for any $x \in \H$, the projections of $\gamma_1$ and $\gamma_2$ on 
            $(T_x \H)^\bot$ are linearly independent;
    \item[$(ii)$] for any $\theta_1,\theta_2 \geq 0$, there exists
            a neighborhood $\VV$ of $\H$ and a function $\psi:\VV \to [0,+\infty)$ such that
    \begin{enumerate}
        \item[$(a)$] $\psi(x)=0$ if $x \in \H$, $\psi(y) >0$ if $y \in \domeg\setminus \H$;
        \item[$(b)$] $\psi$ is Lipschitz continuous in $\VV$ and $C^1$ in $\VV\setminus \H$;
        \item[$(c)$] for any $i=1,2$ and $y\in \VV\cap\tMan{N-1}_i$,
        $ D_x \psi (y)\cdot \gamma_i (y)\geq 0\,;$
        \item[$(d)$] for any $y\in\VV\setminus\H$,
            $ D_x \psi (y)\cdot\big(\theta_1\gamma_1 (y)+\theta_2 \gamma_2 (y)\big) \geq 0\;.$
    \end{enumerate}
    \end{enumerate}
    Then, any \usc viscosity subsolution of the above oblique derivative problem is a stratified
    subsolution of the problem associated with 
    $$\F^{N-1}(x,t,(p_x,p_t))= \sup \Big\{\theta_3 p_t
    +\big(\theta_1\gamma_1 + \theta_2\gamma_2 -\theta_3 \gb^x\big) \cdot p_x -\big(\theta_3 \gl +\theta_1
    g_1+\theta_2 g_2\big)\Big\}\;\hbox{on }\H\times(0,\Tf)\; ,$$ 
    where the supremum is taken on all $(\gb,0,\gl) \in \BCL (x,t)$ such that there exists
    $\theta_1,\theta_2,\theta_3 \in [0,1]$ such that $\theta_1+\theta_2+\theta_3=1$ and
    $\theta_1\gamma_1 + \theta_2\gamma_2 -\theta_3 \gb^x\in T_{x}\H$.

    As a consequence, we have a comparison result for classical viscosity sub and supersolution of
    the oblique derivative problem and therefore a uniqueness result for this problem.
\end{proposition}

In order to clarify the statement of this result, let us comment it in view of the example in
dimension $2$ treated at the begining of the section.

Two main ingredients were important in this $2$-d example. First, we had to solve \eqref{def:pdelta1}
which turns out to be a linear system in $p_\delta$, solvable provided $\gamma_1$ and $\gamma_2$ are
linearly independent. In the general case we assume a little bit more---but we are in a slightly
more complicated framework---because we want $p_\delta$ to be in $(T_x \H)^\bot$, a $2$-dimensional
vector space. Therefore, we require that the projections of $\gamma_1$ and $\gamma_2$ on
this vector space are linearly independent.

Next ingredient is related to function $\psi$, which plays the role of a distance function. More
precisely, we have in mind a distance to $\H$, under the form $\psi(y)=\tilde d (y,\H)$. Of course,
the most usual distance cannot satisfy all our requirements and therefore we need a special distance
$\tilde d(\cdot)$. Like in the $2$-d case, we replace the norm of $y$ by $\psi(y):=(Ay\cdot
y)^{1/2}$ where $A$ actually depends on $\theta_1, \theta_2$. We immediately point out that, if
$\psi$ is not $C^1$ on $\H$, $\psi^2$ is at least $C^1$ on $\VV$.

The formulation we propose is the most intrinsic one. One could think that it is possible to reduce
the proof to a simple flat situation like in $\R^2$ through a change of variable. However, this
change of coordinates interferes in a complicated way with the condition on the directions
$(\gamma_1,\gamma_2)$. Therefore, on one hand, this result indicates what is needed to show that a
classical viscosity subsolution is a stratified subsolution and, on the other hand, the above
example shows how to check the assumptions.

\begin{proof}[Sketch of proof] Since the proof of the result follows the $2$-d case, let us just
    mention the adaptations: we argue locally around $(x,t) \in \H\times(0,\Tf)$ and use the
    function
    $$ (y,s) \mapsto u(y,s)-\phi(y,s) -p_\delta\cdot y -\frac{[\psi (y)]^2}{\e^2}\; ,$$
    where $p_\delta \in (T_x \H)^\bot$ solves 
    $$\gamma_i(x)\big(D_x\phi(x,t)+p_\delta\big)=g_i(x,t) \quad \hbox{for  }i=1,2\; .$$
    With the assumptions on $\psi$ and $\gamma_1,\gamma_2$, the proof readily follows the arguments
    of the $2$-d case.
\end{proof}

\subsection{The Dupuis-Ishii configurations}
\index{Boundary conditions!oblique derivative}\index{Neumann and oblique derivatives problems!Dupuis-Ishii configurations}

In this section, we treat the Dupuis-Ishii configurations mentioned at the beginning of
Section~\ref{sec:Neumann}. We briefly recall that the first configuration corresponds to a smooth
direction of reflection in a stratified domain while in the second case, non-smooth directions of
reflection are considered.

In both configurations, we recall that we assume \HSBC holds, see
page~\pageref{page:simplifications}. We also recall that treating the initial data $u=u_0$ in
\eqref{pb:neumann.initial} is not a problem, as shown in Proposition~\ref{init-Neumann}.

\bigskip

\noindent\textbf{Configuration I ---}
\emph{Smooth reflections in a stratified domain.}

The key assumption here is \Hgam[\domeg], hence $\gamma$ is Lipschitz continuous and $g$ is continuous
on $\domeg\times [0,\Tf]$.

\begin{proposition}\label{prop:DuIs1}
    Assume that \HSBC and \Hgam[\domeg] hold. Then, any \usc viscosity subsolution of
    \eqref{standardHJB}-\eqref{NP} is also a stratified subsolution of
    \eqref{standardHJB}-\eqref{NP}. Therefore, a comparison result holds for viscosity subsolutions
    $u$ and supersolutions $v$ of \eqref{standardHJB}-\eqref{NP} 
 \end{proposition}

\begin{proof}
    Let us first notice that the initial data does not cause any problem for the comparison: as a
    consequence of Proposition~\ref{init-Neumann} with $G(x,p)= \gamma(x,0)\cdot p-g(x,0)$, it
    follows that
    $$u(x,0) \leq u_0(x)\leq v(x,0) \quad \hbox{on  }\Omegb\;.$$
    Hence the whole result easily follows from the first part, \ie from the fact that a standard
    viscosity subsolution is a stratified subsolution.

    In order to prove this fact, we first remark that Proposition~\ref{OD-MN} provides the result on
    $\Man{N}$, namely $\F^N(x,t,(D_xu,D_tu))\leq 0$, so that we are left with proving that
    similarly, for any $1\leq k\leq (N-1)$, 
    $$\F^k(x,t,(D_xu,D_tu))\leq 0 \quad \hbox{on }\Man{k}\;.$$
    Here of course,
    $$\F^k(x,t,(p_x,p_t))= \sup \Big\{\theta p_t -\big(\theta \gb^x-
    (1-\theta)\gamma\big)\cdot p_x -\big(\theta \gl +(1-\theta)g\big)\Big\}\;\hbox{on }\Man{k}\;,$$
    where the supremum is taken on all $(\gb,0,\gl) \in \BCL (x,t)$ such that there exists $\theta \in
    [0,1]$ satisfying $\theta \gb^x-(1-\theta)\gamma \in T_{x}\tMan{k}$.

    In order to do so, we argue locally around a point $(\xb,\tb) \in \Man{k}$ and assume \wlg that
    $\Man{k}\cap B((\xb,\tb),r) =((\xb,\tb)+V_k)\cap B((\xb,\tb),r)$
    for some $r>0$, where $V_k$ is a $k$-dimensional vector space. 

    Notice that any point in $\Man{k}$ necessarily belongs also to $\overline{\Man{N}}$. So, we will
    use the inequality $\F^N\leq0$ to get the result by approaching $\Man{k}$ from $\Man{N}$.

    Let $(\xb,\tb)\in \Man{k}\cap \overline{\Man{N}_1}$ where $\Man{N}_1$ is one of the connected components
    of $\Man{N}$ and let $((\xe,\te))_\e$ be a sequence of points in $\Man{N}_1$ converging to
    $(\xb,\tb)$. By choosing possibly a smaller $r$ we have 
    $((\xe,\te)+V_k)\cap B((\xb,\tb),r)\subset \Man{N}_1$ and 
    $$ \F_\e^k(x,t,(D_xu,D_tu))\leq 0 \quad \hbox{on  }((\xe,\te)+V_k)\cap B((\xb,\tb),r)\; ,$$
    where $\F_\e^k$ is defined in the same way as $\F^k$, noticing that $T_{(x,t)}\Man{k}=V_k$.
    Indeed, this inequality is an immediate consequence of the inequality $\F_\e^k\leq \F^N$
    resulting from the fact that $T_{(x,t)}\Man{k} \subset T_{(x,t)}\Man{N}$ by the \TFS property.

    In order to obtain the result, we just have to let $\e$ tend to $0$: the convergence of
    $\F_\e^k$ to $\F^k$ comes from the normal controllability assumption, together with the
    continuity properties of $\gb, \gl$ and $\gamma$; indeed, on one hand, these properties gives the
    (uniform in $\e$) continuity of the $\F_\e^k$ and, on the other hand, if $((\gb^x,-1),0,\gl)\in
    \BCL(\xb,\tb)$ and $0\leq \theta \leq 1$ are such that $(\theta \gb^x-(1-\theta)\gamma(\xb), -\theta)\in
    T_{(\xb,\tb)}\Man{k}$, then there exists $((\gb_\e^x,-1),0,\gl_\e)\in \BCL(\xe,\te)$ and $0\leq
    \theta_\e \leq 1$ such that $((\gb_\e^x,-1),0,\gl_\e)\to (\gb^x,-1),0,\gl)$, $\theta_\e \to \theta$ and
    $(\theta_\e \gb_\e^x-(1-\theta_\e)\gamma(\xe,\te), -\theta)\in  (\xe,\te)+V_k$.

    We end up with the desired inequality $\F^k(\xb,\tb,(D_xu,D_tu))\leq0$ which completes the
    proof.
\end{proof}

\begin{remark}
    It is worth pointing out that the above proof is valid also for $(x,t)$-dependent stratifications and
    directions of reflections---which was not the case in the papers of Dupuis and Ishii. Moreover,
    the proof is rather simple and it does not require much assumptions. 
\end{remark}

\bigskip

\noindent\textbf{Configuration II ---}
\emph{Discontinuous reflections in a stratified domain.}

\medskip

As in Section~\ref{subsec:codim.one.neumann} in the $\R^N$-case, we assume here that $\Omegb\times
[0,\Tf)$ is a stratified domain with time-independent stratification with $\Omega\times
\R=\tMan{N}\times\R$. However, the boundary stratification is general here: 
$$\domeg=\bigcup_{k=0}^{N-1}\tMan{k}\;,$$
and we assume \Hgamma, so that $\Man{N}=\partial\Omega\times\R=\cup_{i\in I^N} \tMan{N-1}_i \times \R$ and on each 
$\Man{N-1}_i$, we have an oblique derivative boundary condition
$$ \frac{\partial u}{\partial \gamma_i}=g_i(x,t) \quad \hbox{on }\tMan{N-1}_i \times (0,\Tf)\; ,$$
where the $\gamma_i$ is a Lipschitz continuous function, satisfying \eqref{assump:NP} and $g_i$ is a continuous function
on $\tMan{N-1}_i \times [0,\Tf)$.

The statement of the result will remind to the reader that of Proposition~\ref{OD-gamma-disc}.

\begin{proposition}\label{OD-gamma-omega-disc}
    We assume that \HSBC and \Hgamma hold. Moreover, we assume that 
    for any $k=0,..,(N-1)$,
    \begin{enumerate}
    \item[$(i)$] for any $x \in \tMan{k}$ and $0<\delta\ll1$, 
        there exists a unique solution $p_\delta\in (T_x \tMan{k})^\bot$ to the system
        $$ \gamma_i\cdot p_\delta =g_i(x,t)+\delta\;, \quad i\in I(x):=
            \{i:\ x\in  \overline{\Man{N-1}_i}\}\,;$$
    \item[$(ii)$] for any $(\theta_i)_{i\in I(x)}$ where $0\leq\theta_i\leq1$ for any $i$, 
        there exists a neighborhood $\VV$ of $\tMan{k}$ and a function $\psi:\VV \to [0,+\infty)$ such that
        \begin{enumerate}
            \item[$(a)$] $\psi(x)=0$ if $x \in \tMan{k}$, $\psi(y) >0$ if $y \notin \tMan{k}$;
            \item[$(b)$] $\psi$ is Lipschitz continuous in $\VV$ and $C^1$ in $\VV\setminus
                \tMan{k}$;
            \item[$(c)$] for any $y\in \VV\cap\tMan{N-1}_i$, $D_x \psi (y)\cdot \gamma_i (y)\geq 0$;
            \item[$(d)$] for any $y\in\VV\setminus\tMan{k}$, $D_x \psi (y)\cdot \left(\sum_{i\in I(x)}\theta_i\gamma_i
                (y)\right) \geq 0$.
        \end{enumerate}
    \end{enumerate}
    Then, any \usc viscosity subsolution of the above oblique derivative problem is a stratified
    subsolution of the problem associated to 
    $$\F^{k}(x,t,(p_x,p_t))= \sup \left\{\bar \theta p_t
    +\left(\gamma -\bar \theta \gb^x\right) \cdot p_x -(\bar \theta \gl + g)\right\}\;\hbox{on }\Man{k}\; ,$$ 
    where $$\gamma :=\sum_{i\in I(x)}\theta_i\gamma_i (x)\; ,\; \;  g :=\sum_{i\in I(x)}\theta_i g_i (x,t)$$
    and $(\theta_i)_{i\in I(x)}$ is such that $\theta_i \geq 0$  for all $i$. The supremum is taken
    on all $(\gb,0,\gl) \in \BCL (x,t)$ and $(\theta_i)_{i\in I(x)},\bar \theta$ such that $\bar \theta
    + \sum_{i\in I(x)}\theta_i=1$ and $ \gamma-\bar \theta \gb^x\in T_{(x,t)}\Man{k}$.

    As a consequence, a comparison result holds for classical viscosity sub and supersolution of
    the oblique derivative problem, which implies also a uniqueness result for this problem.
\end{proposition}

The comment we could make after the statement of Proposition~\ref{OD-gamma-omega-disc} are the same as those after Proposition~\ref{OD-gamma-disc}, so we drop them and we leave the easy proof of Proposition~\ref{OD-gamma-omega-disc} to
the reader. More interesting are examples which we consider now.

\subsection{Applications to domains with corners}

\noindent\textbf{A. --- Corners in the plane.}

We start by a standard evolution problem in $2$-d described by
Fig.~\ref{fig:neuman.standard}

\begin{figure}[!h]
    \begin{center}
    \includegraphics[width=0.9\textwidth]{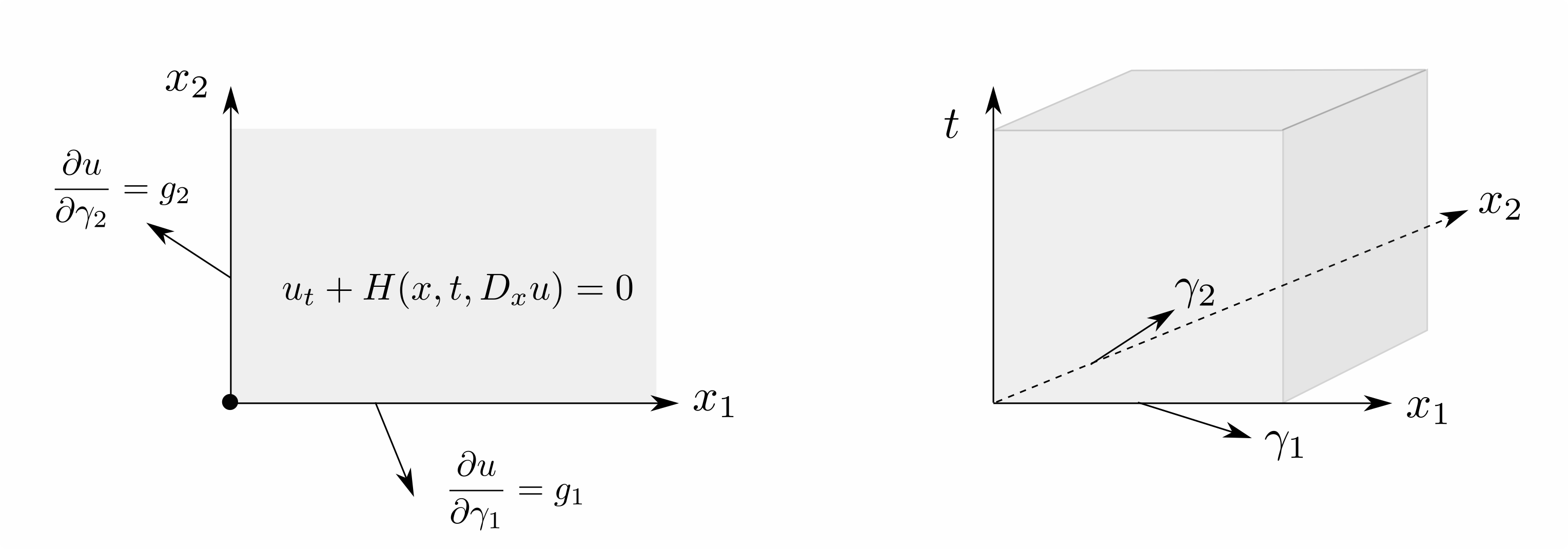}
    \caption{Standard Neuman problem with corner}
    \label{fig:neuman.standard}    
    \end{center}
\end{figure}

Here we have
$$\begin{aligned}
    \Man{3} &= \{(x_1,x_2);\ x_1>0,\ x_2>0\}\times (0,\Tf)\;,\\
    \Man{2} &= \{(x_1,x_2);\ x_1=0,x_2>0 \hbox{ or }x_1>0,x_2=0\}\times (0,\Tf)\;
    ,\\
    \Man{1} &= \{(0,0)\} \times (0,\Tf)\; .
\end{aligned}$$

Of course the analysis of the previous section gives the stratified formulation
on all the boundary except on $\Man{1}$, i.e at the points $((0,0),t)$ for $t\in (0,\Tf)$, which
require a specific treatment. 

For $\Man{1}$, the answer is given by the following result in which we denote by 
$\BCL$ the set of dynamic-discount factor and cost related to $H$. We also point out that, in order to simplify, we argue as if $\gamma_1, \gamma_2,g_1,g_2$ were constants but the reader can check that all the arguments work if they are continuous functions of $x$ and $t$.
 
\begin{proposition}\label{ODCorner-M1}
    We assume that 
    \begin{enumerate}
        \item[$(i)$] either $\gamma_1\cdot e_1=\gamma_2 
\cdot e_2 = 0$ 
\item[$(ii)$] or $\gamma_1\cdot e_1$, $\gamma_2 
\cdot e_2$ have the same strict sign and $\det(\gamma_1,\gamma_2)<0$.
\end{enumerate}
If $u$ is a viscosity 
subsolution of the above oblique derivative problem, it is a stratified subsolution of the problem with
$$\F^1((p_x,p_t))= \sup \left\{\theta_3 p_t -(\theta_3 \gl -\theta_1 g_1-\theta_2 g_2)\right\}\;\hbox{on }\Man{1}\; ,$$
where the supremum is taken on all $(\gb,0,\gl) \in \BCL (0,t)$ such that there exists
    $\theta_1,\theta_2,\theta_3 \in [0,1]$ such that $\theta_1+\theta_2+\theta_3=1$ and $\theta_1\gamma_1 + \theta_2\gamma_2 -\theta_3 \gb^x=0$.
\end{proposition}

\begin{proof}
    Instead of trying to apply Proposition~\ref{OD-gamma-omega-disc} by building only a function
    $\psi$, and since we have not really proved this proposition, we are going to provide the proof
    in the particular case of Proposition~\ref{ODCorner-M1}.  But the reader will notice that the
    important point in the proof below is to build a matrix $A$ in order that $\psi(y):=[Ay\cdot
    y]^{1/2}$ satisfies the requirements of Proposition~\ref{OD-gamma-omega-disc}.

    Let $\phi$ be a $C^1$-function on $\R$ and $t$ be a strict local maximum
    point of the function $s \mapsto u(0,s)-\phi(s)$. We have to show that, if
    $\theta_1,\theta_2,\theta_3, \gamma_1 ,\gamma_2 , \gb$ satisfy the property
    which is required in Proposition~\ref{ODCorner-M1}, then $$\theta_3
    \phi'(t) -(\theta_3 \gl -\theta_1 g_1-\theta_2 g_2)\leq 0\; .$$ It is worth
    pointing out that we can do that only if $(\gb,0,\gl)$ is in the interior of
    $\BCL(0,t)$, a point that we will use in the proof.

To do so, we first construct $p_\delta$ such that
\begin{equation}\label{def:pdelta2} 
p_\delta\cdot \gamma_1=g_1+\delta\quad , 
\quad p_\delta\cdot \gamma_2=g_2+\delta\;,
\end{equation}
notice that such a $p_\delta$ exists 
because of the assumptions on $\gamma_1,\gamma_2$.

Next, we introduce the function
$$ (y,s) \mapsto u(y,s)-\phi(s) -p_\delta\cdot y -\frac{Ay\cdot y}{\e^2}\; .$$
where $A$ is a symmetric, positive definite matrix $A$. Additional properties on
$A$ will be needed and described all along the proof. At the end, we will
show that such a matrix exists.

This function has a maximum point at $(\xe,\te)$ and $(\xe,\te)\to (0,t)$ as $\e \to 0$ by the maximum point property of $(0,t)$. Now we examine the different possibilities: if $\xe =((\xe)_1,0)= (\xe)_1 e_1$ with $(\xe)_1\geq 0$ and $e_1=(1,0)$, we have
$$\begin{aligned}
    \big[p_\delta+\frac{2A\xe}{\e^2}\big]\cdot \gamma_1 & =  
 g_1+\delta + \frac{2\xe}{\e^2}\cdot A\gamma_1\\
& = g_1+\delta + \frac{2(\xe)_1e_1}{\e^2}\cdot A\gamma_1 \; .
\end{aligned}$$
Hence if $A\gamma_1\cdot e_1 \geq 0$, the inequality ``$
\displaystyle \frac{\partial u}{\partial \gamma_1}\leq g_1$'' cannot hold.

In the same way, if $\xe =(0, (\xe)_2)=(\xe)_2e_2$, $(\xe)_2\geq 0$ and
$e_2=(0,1)$, the inequality ``$\displaystyle \frac{\partial u}{\partial
    \gamma_2}\leq g_2$'' cannot hold provided $A\gamma_2\cdot e_2 \geq 0$.

Therefore, wherever $\xe$ is, the $H$-inequality holds and, since $(\gb,0,\gl)$ is in the interior of $\BCL(0,t)$, for $\e$ small enough, $(\gb,0,\gl) \in \BCL(\xe,\te)$  and we have
$$ \phi'(\te) -\gb^x \cdot[p_\delta+\frac{2A\xe}{\e^2}]\leq l \; .$$
Now we examine the $\gb^x$-term, remarking that $\theta_3$ cannot be $0$
\begin{align}
 -\gb^x \cdot \frac{2A\xe}{\e^2}= & -\frac1{\theta_3}(\theta_1\gamma_1+\theta_2\gamma_2)\cdot \frac{2A\xe}{\e^2}\\
 = & - \frac1{\theta_3}(\theta_1A\gamma_1+\theta_2 A\gamma_2)\cdot
 \frac{2\xe}{\e^2}\; ,
 \end{align}
 (recall that, being a symmetric matrix, the transpose of $A$ is $A$ itself).
Since we want this term to be positive for any $\xe=((\xe)_1,(\xe)_2)$ with
 $(\xe)_1,(\xe)_2 \geq 0$, we have to require that all the coordinates of the
 vector $\theta_1A\gamma_1+\theta_2 A\gamma_2$ be negative.
 
If these properties hold true, we end up with
 $$\phi'(\te) -\gb ^x \cdot p_\delta \leq l\; .$$
 Letting $\e$ tend to $0$, and using a convex combination with \eqref{def:pdelta2}
 provides the answer, after letting $\delta$ tend to $0$.

It remains to show that such matrix $A$ exists under the conditions of 
Proposition~\ref{ODCorner-M1}. We point out that this matrix may depend on the
convex combination, hence on $\theta_1, \theta_2,\theta_3$ since the above proof
is done for any fixed such convex combination.

We recall that the three conditions are
$$A\gamma_1\cdot e_1 \geq 0\quad , \quad A\gamma_2\cdot e_2 \geq 0\; ,$$
$$\theta_1A\gamma_1+\theta_2 A\gamma_2\leq 0\; ,$$
this last condition meaning that all the components of the vector are negative.

Looking at these conditions, a natural choice would be
$$ A\gamma_1= -\lambda_1 e_2\quad \hbox{and}\quad A\gamma_2=-\lambda_2e_1\; .$$
where $\lambda_1, \lambda_2$ are non-negative constants which have to be chosen properly, or equivalently
$$ A^{-1}e_1= - (\lambda_2)^{-1} \gamma_2\quad\hbox{ and }\quad A^{-1}e_2= - (\lambda_1)^{-1} 
\gamma_1\; .$$
Therefore
$$ A^{-1}=
\left(\begin{array}{cc}  
-(\lambda_2)^{-1} \gamma_{2,1} & -(\lambda_1)^{-1}\gamma_{1,1} \\ 
-(\lambda_2)^{-1} \gamma_{2,2} & - (\lambda_1)^{-1}\gamma_{1,2}
\end{array}\right)\; .$$
In order that $A$ satisfies the required condition to be a symmetric, positive definite matrix, we have just to check that $ A^{-1}$
satisfies these properties. Recalling that $\gamma_1\cdot e_2=\gamma_{1,2}<0$ and $\gamma_2\cdot e_1=\gamma_{2,1}<0$ by
the oblique derivatives conditions, this leads to the following conditions
\begin{enumerate}
    \item[$(i)$] $ A^{-1}$ is symmetric if either $\gamma_{1,1}=\gamma_{2,2}=0$ or
$\gamma_{1,1},\gamma_{2,2}$ have the same strict sign. Then we can choose
$\lambda_1=|\gamma_{1,1}|$ and $\lambda_2 =|\gamma_{2,2}|$.
\item[$(ii)$] The trace of $A$ is non-negative since $\gamma_{2,1},\gamma_{1,2}<0$ by the conditions on the directions of reflection.
\item[$(iii)$] $\det(A^{-1})= (\lambda_1\lambda_2)^{-1} \det(\gamma_2,\gamma_1)>0$ by assumption.
\end{enumerate}
Hence we can conclude if one of the two conditions holds\\
1. $\gamma_{1,1}=\gamma_{2,2}=0$ with $A=Id$.\\
2. $\gamma_{1,1},\gamma_{2,2}$ have the same strict sign and $\det(\gamma_1,\gamma_2)<0$.\\

In order to investigate the other cases and to show that $A$ does not exist in these cases, we assume (without loss of generality) 
that $\gamma_1 =(\gamma_{1,1},-1)$, $\gamma_2=(-1,\gamma_{2,2})$ and we 
write $A$ as
$$ A=\left(\begin{array}{cc}  \alpha & \beta  \\  \beta  & \gamma
\end{array}\right)\; ,$$
where $\beta$ can be chosen as $0, 1$ or $-1$ since $A$ can be replaced by $\lambda A$ for $\lambda>0$.

The constraint can be written as
$$ \alpha\gamma_{1,1}-\beta \geq 0\; ,$$
$$ -\beta + \gamma\gamma_{2,2} \geq 0\; ,$$
 $$ \theta_1(\alpha\gamma_{1,1}-\beta)+\theta_2(-\alpha +\beta \gamma_{2,2}) \leq 0\; ,$$
 $$ \theta_1(\beta\gamma_{1,1}-\gamma)+\theta_2(-\beta + \gamma\gamma_{2,2}) \leq 0\; .$$

We begin with the case when $\gamma_{1,1} \geq 0$, $\gamma_{2,2}\leq 0$.
In this case, the (necessary) choice $\beta=-1$ yields
$$ \alpha\gamma_{1,1}+1 \geq 0\; ,$$
$$1 + \gamma\gamma_{2,2} \geq 0\; ,$$
 $$ \theta_1(\alpha\gamma_{1,1}+1)+\theta_2(-\alpha - \gamma_{2,2}) \leq 0\; ,$$
 $$ \theta_1(-\gamma_{1,1}-\gamma)+\theta_2(1 + \gamma\gamma_{2,2}) \leq 0\; .$$
The first constraint gives no limitation on $\alpha$, while the second one imposes (a priori) $\gamma$ to be small enough. For the two next ones we recall that $\theta_3 \gb= \theta_1\gamma_1 + \theta_2 \gamma_2$ and therefore
$$\theta_3 \gb_1 = \theta_1\gamma_{1,1} - \theta_2 \quad ,\quad
\theta_3 \gb_2 = - \theta_1 + \theta_2 \gamma_{2,2}<0\; .$$
Hence the two last constraints can be written as
$$ \alpha \gb_1 -\gb_2 \leq 0 \; ,$$
 $$ -\gb_1 + \gamma \gb_2 \leq 0\; .$$
Clearly one can conclude only if $\gb_1<0$ by choosing $\alpha$ large and with
$$ \frac{\gb_1}{\gb_2} \leq \gamma \leq - \frac{1}{\gamma_{2,2}}\; .$$
We have indeed the existence of such $\gamma$ since
 $$ \frac{\gb_1}{\gb_2} \leq - \frac{1}{\gamma_{2,2}}$$
 is equivalent to $\det (\gb,\gamma_2)\leq 0$ and $\displaystyle \det
 (\gb,\gamma_2)=\frac{\theta_1 }{\theta_3}\det(\gamma_1,\gamma_2)\leq 0$.
 
 But, given $\gamma_1, \gamma_2$, in order to have $\gb_1< 0$ for any choice of the convex combination, the only possibility is $\gamma_{1,1}=0$. And the same conclusion holds in the case
$\gamma_{1,1} \leq 0$, $\gamma_{2,2}\geq 0$.

The proof is then complete.
\end{proof}

\begin{example}We consider the following problem where $Q=(0,1) \times (0,1)$
\begin{equation}
\begin{cases}\label{NPexample}
u_t + a(x,t)|D_x u|=f(x) & \hbox{in $ Q \times (0,\Tf)$} \\
u(x,0) =u_0 (x) & \hbox{in $\Omega$}\\
\displaystyle \frac{\partial u}{\partial n_i} =g_i (x,t) & \hbox{on $\partial Q_i \times (0,\Tf)$\; ,}
\end{cases}
\end{equation}
where $\partial Q_1 = (0,1)\times \{0\}$, $\partial Q_2 = \{1\} \times (0,1) $, , $\partial Q_3 = (0,1)\times \{1\}$, ,$\partial Q_4 = \{1\} \times (0,1)$ and $n_i$ is the exterior unit normal vector to $\partial Q_i$.

If $a$ is a Lipschitz continuous function (in particular in $x$) satisfying $a(x,t) >0$ on $\overline Q \times [0,\Tf]$, and $u_0, f, g_1,\cdots g_4$ are continuous, there exists a unique viscosity solution of this problem which coincides with the stratified solution. This result is a straightforward consequence of the former results which shows that the notions of viscosity solutions and stratified solutions are the same. It is worth remarking on this example that, despite we did not insist above on that point, the Hamiltoniant $\F^1,\F^2$ satisfy the right conditions: indeed these Hamiltonians fullfill the required continuity assumptions because in the above convex combinations like $\theta_1+\theta_2+\theta_3=1$ and $\theta_1\gamma_1 + \theta_2\gamma_2 -\theta_3 \gb^x=0$, $\theta_3$ is bounded away from 0.

In the present exemple, on $\partial Q_i$
$$ \F^2(x,t,(p_x,p_t))=\max(\theta(p_t -a(x,t)v\cdot p_x -f(x))+(1-\theta)(n_i \cdot p_x -g_i))\; ,$$
the maximum being taken on all $|v|\leq 1$ and $\theta \in [0,1]$ such that $[\theta a(x,t)v - (1-\theta)n_i ]\cdot n_i=0$

Writing $v= v^\bot + v^\top$, where $v^\bot$ is the normal part of $v$ (i.e. the part which is colinear to $n_i$) and $v^\top$ the tangent part, we have $\theta a(x,t)v^\bot\cdot n_i = (1-\theta)$ and we take divide by $\theta$ to have 
$$ \F^2(x,t,(p_x,p_t))=\max_{|v^\top|^2 + |v^\bot|^2=1}(p_t -a(x,t)v^\top \cdot p_x -f(x)+a(x,t)v^\bot\cdot n_i g_i)\; .$$
On an other hand, at $x=0$, a simple computation gives
$$ \F^1(0,t,p_t)=\max(p_t -f(x)- g_1; p_t -f(x)- g_4)\; .$$
\end{example}

\begin{remark}We do not know if the conditions given in Proposition~\ref{ODCorner-M1} are
optimal or not. Clearly they are stronger than those given in Dupuis and Ishii \cite{DuIs1,DuIs2}
inspired by those of Harrison and Reiman \cite{HaRe} and Varadhan and Williams \cite{VaWi}. Maybe a
different choice of test-function, namely the term $\e^{-2}(Ay\cdot y)$, could lead to more general
cases but we have no idea how to build such a function which necessarily will be $C^1$ but not $C^2$
at $0$. \end{remark}

\bigskip

\noindent\textbf{B. --- The $\R^N$ case.}

\medskip

Of course, in $\R^N$, there exists a lot of possibilities and we are going to investigate the following three situations\\
\noindent{\bf Case 1: a simple $1$-dimensional corner} 
$$\begin{aligned}
    \Man{N+1} &= \{(x_1,\cdots, x_N);\ x_{1}>0, x_2>0\}\times (0,\Tf),\\
    \Man{N}\quad &= \{(x_1,\cdots, x_N);\ x_{1}=0,x_2>0 \hbox{ or
        }x_{1}>0,x_2=0\}\times (0,\Tf)\; ,\\
  \Man{N-1} &= \{(x_1,\cdots, x_N);\ x_{1}=0,x_2=0\} \times (0,\Tf)\;.
  \end{aligned}$$
\noindent{\bf Case 2: a simple discontinuity in the oblique derivative boundary 
condition} 
$$\begin{aligned} 
    \Man{N+1} &= \{(x_1,\cdots, x_N);\ x_2>0\}\times (0,\Tf), \\
    \Man{N}\quad &= \{(x_1,\cdots, x_N);\ x_{1}\neq 0,x_N=0\}\times (0,\Tf)\; ,\\
    \Man{N-1} &= \{(x_1,\cdots, x_N);\ x_{1}= 0,x_2=0\} \times (0,\Tf)\;.
\end{aligned}$$
\noindent{\bf Case 3: a multi-dimensional corner}
$$\begin{aligned} 
    \Man{N+1} &= \{(x_1,\cdots, x_N);\ x_1>0 \cdots x_N>0\}\times (0,\Tf), \\
    \Man{N}\quad &= \bigcup_i \{(x_1,\cdots, x_N);\ x_1\geq 0 \cdots x_N\geq 0,
        x_i=0\}\times (0,\Tf)\; ,\\
    \Man{N-1} &= \{(x_1,\cdots, x_N);\ x_{N-1}=0,x_N=0\} \times (0,\Tf)\;.
\end{aligned}
$$

In each case, the question is: when is the classical notion of subsolution 
equivalent to the stratification formulation?

The answer is simple in the two first cases. Let us write 
$$ \gamma_1=(\gamma^{(1)}_1, \gamma^{(1)}_2, \cdots,\gamma^{(1)}_N) \quad 
\hbox{and}\quad \gamma_2=(\gamma^{(2)}_1, \gamma^{(2)}_2, \cdots,
\gamma^{(2)}_N)\; ,$$
and introduce the two vectors of $\R^2$ 
$$ \tilde \gamma_1=(\gamma^{(1)}_1, \gamma^{(1)}_2) \quad \hbox{and}\quad 
\tilde \gamma_2=(\gamma^{(2)}_1, \gamma^{(2)}_2)\; .$$

The result is
\begin{proposition}In Case 1 and 2, the classical viscosity formulation and the stratified formulation are equivalent if $\tilde \gamma_1,\tilde \gamma_2$ satisfy the condition of Proposition~\ref{ODCorner-M1} in Case~1 and Proposition~\ref{ODFlat-M1} in Case~2.
\end{proposition}

\begin{proof}In Case 1, we have to show that a viscosity subsolution $u$ is also a stratified subsolution on $\Man{N-1}$. To do so, we denote any $x\in\R^N$ by $(x_1,x_2,x')$ where $x'=(x_3,\cdots,x_N)$. 

If $(\xb,\tb)\in \Man{N-1}$ is a maximum point of $x' \mapsto u(0,0,x',t)-\phi(x',t)$ where $\phi$ is a smooth function, we have to show that, if we have a convex combination $(- \theta_1 \gamma_1 -\theta_2 \gamma_2 + \theta_3 \gb^x,-\theta_3) \in T_{(\xb,\tb)}\Man{N-1}$ with $(\gb,0,\gl) \in \BCL(\xb,\tb)$ and $\gb=(\gb^x,-1)$. Then
$$ \theta_3 \phi_t +(\theta_1 \gamma_1 + \theta_2 \gamma_2 - \theta_3 \gb^x)\cdot D_x \phi(\xb',\tb)\leq - \theta_1 g_1 - \theta_2 g_2 + \theta_3 l\; .$$

As in the proof of Proposition~\ref{ODCorner-M1}, we introduce $p_\delta$ such that
$$ p_\delta\cdot \gamma_1=g_1+\delta\quad , \quad p_\delta\cdot \gamma_2=g_2+\delta\; ,$$
and the function
$$ (y,s) \mapsto u(y,s)-\phi(y',s) -p_\delta\cdot y -\frac{A\tilde y\cdot \tilde y}{\e^2}\; ,$$
where $\tilde y=(y_1,y_2)$ and $A$ is a $2\times 2$ symmetric, positive definite matrix.

It is clear on this formulation that, only the $\tilde y$ terms plays a real role and we are in the same situation as in $\R^2$. This explains the statement of the result and shows that the proof for Case 2 follows from the same arguments.
\end{proof}

\begin{example}A standard example for Case~2 is the case when we look at an oblique derivative problem
in a smooth domain $\Omega \subset \R^N$ whose boundary is splitted into three components
$$ \domeg = \domeg_1 \cup \domeg_2 \cup \Gamma\; ,$$
where $\domeg_1,\domeg_2$ are smooth $(N-1)$-dimensional manifolds and $\Gamma$ a smooth $(N-2)$-dimensional manifold. The idea is to have the oblique derivative boundary condition
$$\frac{\partial u}{\partial \gamma_i} =g_i\quad \hbox{on $\domeg_i\times (0,\Tf)$}\; ,$$
for $i=1,2$.

The question is when the classical viscosity solution coincides with the stratified one and therefore is unique?

To answer to this question is not completely obvious since we have to apply the above result for Case~2 in the right way on $\Gamma$. To do so, we consider $x\in \Gamma$ and we introduce two unit vectors: $n$ the unit outward normal to $\domeg$ at $x$ and $r \in T_x \domeg$ a unit vector which is normal to $T_x\Gamma$ and which is pointing toward $\Omega_1$.

With these notations, the answer to the above question is yes if the determinant
$$ \left|\begin{array}{cc}\gamma_1\cdot r & \gamma_2\cdot r \\-\gamma_1\cdot n & -\gamma_2\cdot n \end{array}\right| \leq 0\; .$$
\end{example}

For Case 3, we introduce the $N\times N$-matrix $\Gamma$ whose columns are given by $\gamma_1, \gamma_2,\cdots \gamma_n$ and we formulate the
\begin{proposition}In Case 3, the classical viscosity formulation and the stratified formulation are equivalent if there exists a $N\times N$-diagonal matrix $D$ with strictly positive diagonal terms such that $\Gamma.D^{-1}$ is a symmetric, negative definite matrix.
\end{proposition}

\begin{proof}The proof follows along the arguments of the proof of Proposition~\ref{ODCorner-M1}: the key (and only) point is to find a symmetric, positive definite matrix $A$ such that
$$ A\gamma_i = -d_i e_i\quad \hbox{with  }d_i >0,\ \  \hbox{for any  }1\leq i \leq N \; .$$
 This property can be written as $A.\Gamma=-D$ and therefore $A=-D\Gamma^-1=-(\Gamma. D^{-1})^{-1}$. The assumption ensures the existence of $A$.
\end{proof}

This result can, of course, be extended to the case of more general convex domains like
$$ \Omega := \bigcap_i \{x:\ n_i \cdot x < q_i\}\; ,$$
with a direction of reflection $\gamma_i$ on $\{x:\ n_i \cdot x = q_i\}$ by the
\begin{proposition}
The classical viscosity formulation and the stratified formulation are equivalent if there exists a $N\times N$ symmetric, positive definite matrix $A$ such that
$$ A\gamma_i = d_i n_i\quad \hbox{with  }d_i >0,\ \  \hbox{for any  }1\leq i \leq N \; .$$.
\end{proposition}
\begin{remark}Clearly, as in the case of the $2$-dimensional corner we have no idea 
if these conditions are optimal or not but, at least, they are obviously satisfied if $\gamma_i=n_i$ with $A=Id$ and all $d_i=1$.
\end{remark}

We conclude this section by an open question in the case of a non-convex domain, the model case being in $2$-d
$$ \Omega =\{(x_1,x_2):\ x_1>0\hbox{  or  }x_2>0\}\; ,$$
with normal reflection on the two parts of the boundary, $\{(x_1,0):\ x_1>0\}$ and
$\{(0,x_2):\ x_2>0\}$, or different oblique derivative boundary conditions.

The strategy we follow above clearly fails due to the non-convexity of the domain and, maybe surprisingly, we were unable to obtain any general result in this case (some 
particular cases can, of course, be treated). We do not know if 
this is just a technical problem or if really they are counterexample where Ishii solutions 
are not unique since, otherwise, they coincide with the unique stratified solution.

\section{Mixing the Dirichlet and Neumann problems}

\index{Boundary conditions!mixed}

In this section, we present two very different models mixing Dirichlet and oblique derivative problems on
the boundary:
\begin{enumerate}
    \item[$(i$)] the most standard case is when the boundary $\domeg$ can be
        decomposed as $\partial\Omega_1\cup\partial\Omega_2\cup\H$, where $\partial\Omega_{1/2}$ are
        open subsets of the boundary and $\H$ is a $(N-2)$ submanifold of $\partial\Omega$,
        the boundary condition being of composite type: Dirichlet on $\partial\Omega_1\times(0,\Tf)$
        and oblique derivative on $\partial\Omega_2\times(0,\Tf)$;
    \item[$(ii$)] the second example is the ``Tanker problem'', a far less standard case already
        presented in Section~\ref{sect:tanker}, involving Neumann conditions on
        $\{P_i\}\times(0,\Tf)$, for a collection of points $(P_i)$ in $\R^N$ (the harbors).
\end{enumerate}

We treat both examples in the context of \HSBC, the stratified approach we developped allowing us to
handle these situations almost effortlessly.

\subsection{The most standard case}

Here we still assume that $\Omega$ is a bounded domain and we also assume that the boundary $\domeg$
can be decomposed as $$ \domeg = \domeg_1\cup\domeg_2\cup \H\; ,$$
where, in terms of stratification
\begin{equation}\label{omega:mix}
\tMan{N-1}=\domeg_1\cup\domeg_2\quad\hbox{and}\quad \tMan{N-2}=\H\; .
\end{equation}

The boundary condition we consider is the following:
\begin{align}{}
  u = \varphi \quad\hbox{on  }\domeg_1 \times (0,\Tf)\;, \label{partDir}\\
  \frac{\partial u}{\partial \gamma}  =g \quad\hbox{on  }
  \domeg_2 \times (0,\Tf)\;,\label{partNeu}
\end{align}
where $\varphi$ is a continuous function and $\gamma,g$ satisfy \Hgam[\domeg_2].

One may have in mind two cases depicted in Fig.~\ref{fig:neuman.dir}: either, like on the left side,
$\domeg$ is smooth or, like on the right, we may face a corner.

\medskip

\begin{figure}[!h]
    \begin{center}
    \includegraphics[width=0.4\textwidth]{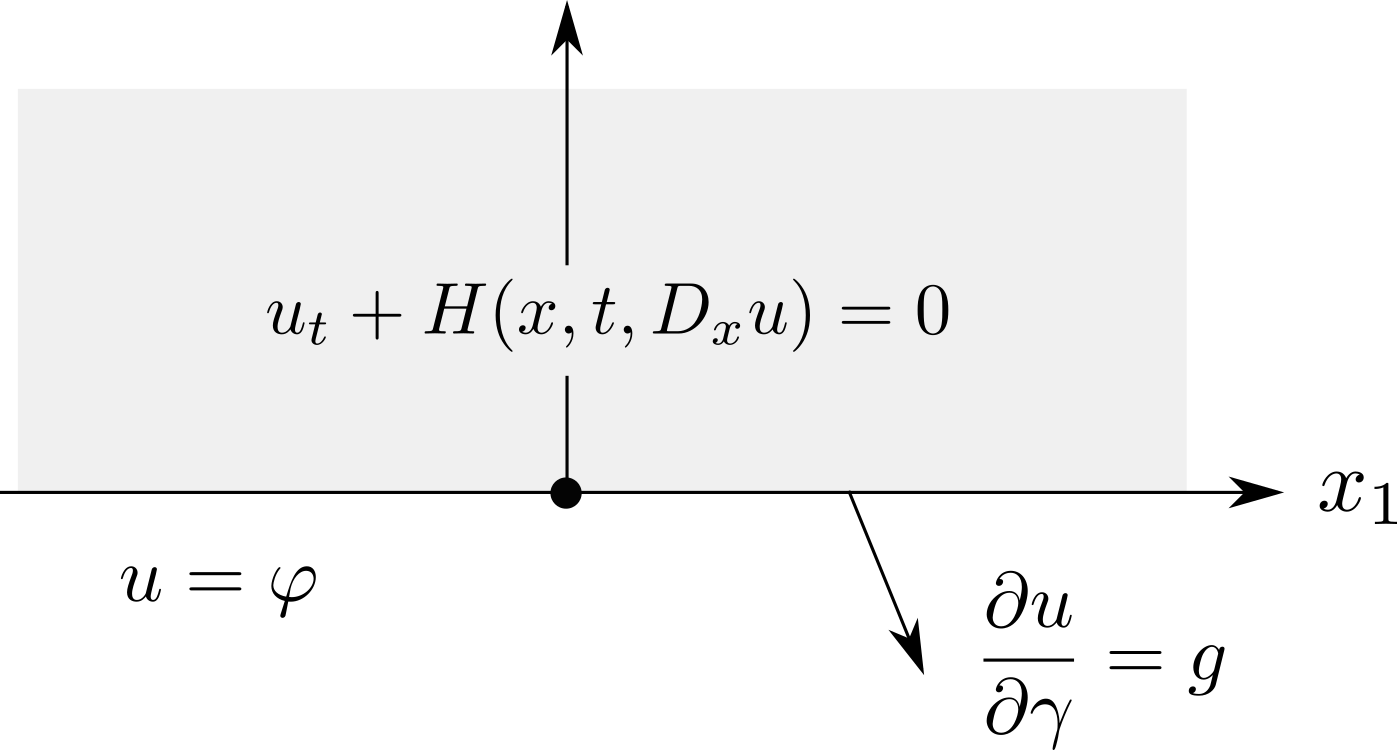}
    \qquad
    \includegraphics[width=0.4\textwidth]{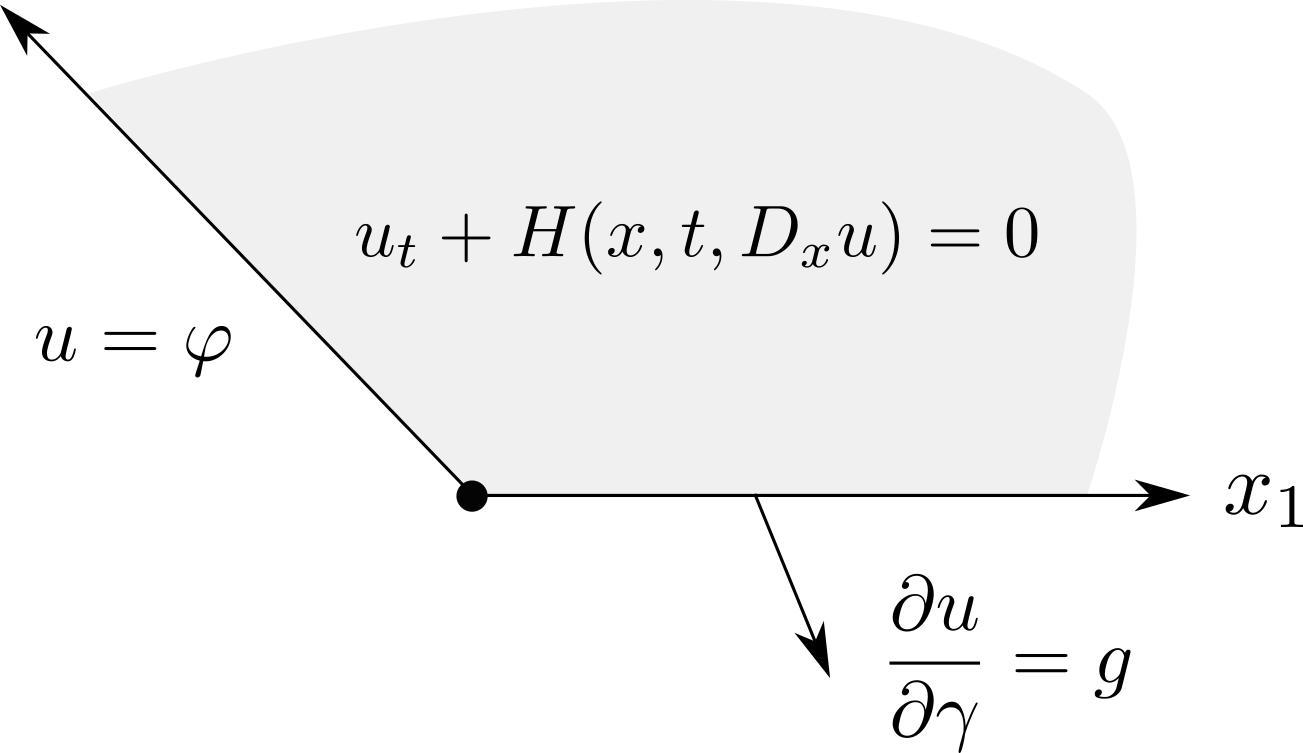}
    \caption{Flat and angular Dirichlet-Neuman problems}
    \label{fig:neuman.dir}    
    \end{center}
\end{figure}

\bigskip

Maybe surprisingly, both cases can be treated in the same way, the main
property needed being the 
\begin{lemma}\label{Dir-ineq-ju}
    Let $\Omega$ be a bounded, stratified domain satisfying \eqref{omega:mix} and assume that $\varphi$ is a continuous function and that
    \HSBC and \Hgam[\domeg_2] hold.
    Then
    $$ u\leq \varphi \quad \hbox{on  }\Man{N-1}= \H\times (0,\Tf)\;.$$
\end{lemma}

\begin{proof}Let $(\xb,\tb) \in \Man{N-1}$; we want to prove the inequality
$u(\xb,\tb)\leq \varphi(\xb,\tb)$.

We first remark that $u(x,t)\leq \varphi(x,t)$ if $x\in \partial \Omega_1$, $t >0$ as a consequence of the results for 
the Dirichlet problem. Hence, if we assume by contradiction that $u(\xb,\tb)> \varphi(\xb,\tb)$ and if we redefine $u$ 
on $\domeg_1$ by introducing
$$ \tilde u (x,t)= \limsup_{\displaystyle  \mathop{\scriptstyle (y,s) \to (x,t)}_{\scriptstyle  y\in \Omega}} u(y,s)\quad 
\hbox{if  } x \in \domeg_1\; ,$$
$\tilde u$ being equal to $u$ otherwise, then $\tilde u$ is still an \usc subsolution of the problem.

Next, following the arguments of the Neumann part, we have
$$ \tilde u (x,t)= \limsup_{\displaystyle  \mathop{\scriptstyle (y,s) \to (x,t)}_{\scriptstyle  y\in \Omega}} \tilde u(y,s)\quad \hbox{if  } x \in \domeg_2,\  t>0\; .$$
With all these properties, the regularization procedure of Section~\ref{sec:regplus-subsol} together with
 normal controllability properties of $H$ implies that the partial
 sup-convolution procedure in the $\Man{N-1}$-direction provides a sequence of functions $(\ue)_\e$ which are Lipschitz continuous in $B((\xb,\tb),r)\cap [\Omegb\setminus \H]\times (0,\Tf)$, but we have a discontinuity at least at $(\xb,\tb)$.
 
 Moreover we have
 \begin{enumerate}
 \item $\ue(x,t)\geq \tilde u(x,t)$ on $B((\xb,\tb),r)\cap \Omegb\times (0,\Tf)$ by construction (the sup-convolution),
 \item $\ue(x,t)\leq \varphi (x,t)+o_\e (1)$ on $B((\xb,\tb),r)\cap[\domeg_1\times (0,\Tf)]$ again by construction since $\tilde u(x,t)\leq \varphi(x,t)$ if $x\in 
 \partial \Omega_1$, 
 \item By using the Lipschitz continuity of $\ue$:
 $$  \limsup_{\displaystyle  \mathop{\scriptstyle (y,s) \to (\xb,\tb)}_{\scriptstyle  
 y\in\Omega\cup\domeg_1\cup\domeg_2 }} \ue(y,s)\leq \varphi(\xb,\tb) \; .$$
 \end{enumerate}
 
 As a consequence of Points 1 and 3, we also have
 $$  \limsup_{\displaystyle  \mathop{\scriptstyle (y,s) \to (\xb,\tb)}_{\scriptstyle  
 y\in\Omega\cup\domeg_1\cup\domeg_2 }} \ue(y,s)\leq \varphi(\xb,\tb) < \tilde u(\xb,\tb)\; .$$
 
 Now we consider the function $\displaystyle (x,t)\mapsto \tilde u (x,t)
 -\frac{|x-\xb|^2}{\e^2}-\frac{|t-\tb|^2}{\e^2}$: for $\e>0$ small enough, this
 function has a maximum point at $(\xe ,\te)$ near $(\xb,\tb)$ and the above
 properties implies:
 
 \noindent $(i)$ $(\xe ,\te) \in \Man{N-1}$,\\
 $(ii)$ $\tilde u (\xe ,\te)\to
 \tilde u (\xb,\tb)$ and $\displaystyle
 \frac{|\xe-\xb|^2}{\e^2}+\frac{|\te-\tb|^2}{\e^2}\to 0 $ as $\e\to 0$,\\
 $(iii)$ assuming without loss of generality that $\tMan{N-1}$ is flat, for any vector $p$ which is normal to $\tMan{N-1}$, $(\xe ,\te)$ is still a maximum
 point of $$(x,t)\mapsto \tilde u (x,t)
 -\frac{|x-\xb|^2}{\e^2}-\frac{|t-\tb|^2}{\e^2}-\frac{p\cdot (x-\xe)}{\e}\; .$$
 Choosing $p$ such that $p\cdot \gamma(\xe,\te)>0$---this is possible since $\gamma(\xe,\te)$ cannot be in $\tMan{N-1}$---and using the normal controllability assumption, it is clear that the viscosity subsolution inequality at $(\xb,\tb)$ cannot hold for $\e$ small enough, giving the desired contradiction.
\end{proof}

As we did separately for the Dirichlet and oblique derivative problems, let us check that under our
hypotheses, the initial condition is of Cauchy type.
\begin{lemma}
    Let $\Omega$ be a bounded, stratified domain satisfying \eqref{omega:mix} and assume that $\varphi,u_0$ are a continuous functions and that
    \HSBC and \Hgam[\domeg_2] hold. If
\begin{equation}\label{eqn:comp-domeg1}
\varphi(x,0)=u_0(x)\quad\hbox{on  } \domeg_1 \times \{0\}\, ,
\end{equation}
and if $u$ and $v$ are respectively an \usc viscosity subsolution and a \lsc supersolution of
    the mixed problem, necessarily $$ u(x,0)\leq u_0(x) \leq v(x,0) \quad \hbox{on }\domeg\;.$$
\end{lemma}

\begin{proof}
    Of course, the difficulty comes from the points of $\domeg \times \{0\}$ which are located on
    $\H$. We only prove the result for the subsolution, the proof for the supersolution being
    analogous.

    If $\xb\in \H$, we want to prove that $u(\xb,0) \leq u_0(\xb)$. Since we are in a stratified
    framework, we can assume that $\domeg_2\subset \{x:\ (x-\xb)\cdot n_2=0\}$, where $n_2\cdot
    \gamma >0$ on $\domeg_2$.  

    For $0<\e \ll 1$, we consider the function 
    $$ (x,t) \mapsto u(x,t)-\frac{|x-\xb|^2}{\e^2}-C_\e t -\e \psi \left(\frac{(x-\xb)\cdot
        n_2}{\e^2}\right)\; .$$ 
    where $C_\e>\e^{-1}$ is a large constant to be chosen later on and $\psi :\R\to \R$ is a $C^1$,
    increasing function such that $\psi(t)=-1$ if $t\leq -1$, $\psi(t)=1$ if $t\geq 1$ and,
    $\psi(0)=0$, $\psi'(0)=1$.

    This function has a maximum point at $(\xe ,\te)$ near $(\xb,0)$ and 
    $(\xe ,\te) \to (\xb,0)$, $u (\xe ,\te)\to  u (\xb,0)$, 
    $\displaystyle \frac{|\xe-\xb|}{\e^2}+C_\e \te \to 0 $ as $\e\to 0$. 

    If $\xe \in \domeg_2$, then the oblique derivative inequality cannot holds since in the viscous
    sense,
    $$ Du \cdot \gamma = \frac{2(\xe-\xb)}{\e^2}\cdot\gamma+\frac1{\e} \psi' (0)n_2\cdot\gamma=
    \frac{o(1)}{\e}+\frac1{\e}n_2\cdot \gamma>0\; ,$$
    for $\e$ small enough. On the other hand, by choosing $C_\e$ large enough, the $H$-inequality
    cannot hold wherever  $(\xe ,\te)$ is. Hence one of the inequalities $u (\xe ,\te)\leq \varphi
    (\xe ,\te)$ or $u (\xe ,\te)\leq u_0  (\xe)$ holds and the conclusion follows by letting $\e$
    tend to $0$.
\end{proof}

The above two lemma give us the final result:

\begin{proposition}\label{Mix-Gen}\emph{--- Comparison for a mixed Dirichlet-Neumann case.}\smsp
Let $\Omega$ be a bounded, stratified domain satisfying \eqref{omega:mix} and assume that $\varphi, u_0$ are continuous functions 
satisfying \eqref{eqn:comp-domeg1} and that
    \HSBC and \Hgam[\domeg_2] hold.
    If $u$ is an \usc viscosity subsolution of the mixed problem, then the function
    $\tilde u : \overline \Omega \times [0,\Tf) \to \R$ defined by $\tilde u
    (x,t)=u(x,t)$ if $x\in \Omega \cup \domeg_2 \cup \H$ and $$ \tilde u (x,t)=
    \limsup_{\displaystyle  \mathop{\scriptstyle (y,s) \to (x,t)}_{\scriptstyle
        y\in \Omega}} u(y,s)\quad \hbox{if  } x \in \domeg_1\cup \H\; ,$$
    is a stratified subsolution of the stratified problem associated to the
    Hamiltonian defined on $\Man{N-1}=\H\times(0,\Tf)$:
    $$\begin{aligned}\mathbb{F}^{N-1} & (x,t,Du) :=\\
    &\max \Big(u-\varphi(x,t)\,;\, \sup \left\{\theta p_t -(\theta \gb^x-
    (1-\theta)\gamma)\cdot p_x -(\theta \gl +(1-\theta)g)\right\}\Big)\;,
    \end{aligned}$$
    where the supremum is taken on all $(\gb,0,\gl) \in \BCL (x,t)$ such that there
    exists $\theta \in [0,1]$ satisfying $\theta \gb^x-(1-\theta)\gamma\in
    T_{x}\tMan{N-1}$.

    As a consequence, up to a modification of the value of the subsolution on $\domeg_1\cup\H$, a
    comparison result holds for the mixed problem and therefore there exists a unique continuous viscosity
    solution of the mixed problem, up to this modification. 
\end{proposition}

The proof of this result is simple since we use both the ingredients for the Dirichlet and Neumann
problems in $\domeg_1$ and $\domeg_2$, the only difficulty being of course to deal with
$\H\times (0,\Tf)=\Man{N-1}$. 
Lemma~\ref{Dir-ineq-ju} provides the inequality $\tilde u\leq \varphi$ on $\Man{N-1}$, while the
other part in $\F^{N-1}$ is obtained by a stability result from ``inside'' $\Man{N}$ in the spirit of
Remark~\ref{MgMk} or Proposition~\ref{IequalS}.
We leave these details to the reader.

\subsection{The tanker problem}\label{sec:solutionTP}

\index{Applications!tanker problem}

As an example where the classical Ishii viscosity solutions formulation cannot be sufficient for
treating singular discontinuities, we  come back to the example given by P.L. Lions in his course at
the Coll\`ege de France, namely the problem \eqref{PL-Tanker}.

At $P_1, P_2, P_L$, one would like to impose Neumann boundary conditions
$$ \frac{\partial u}{\partial n} =g_i (t) \quad \hbox{at $P_i$\; ,}$$
(see Fig.~\ref{fig:tanker}) but such a boundary condition is far from being
classical. However, we can handle it through the stratified formulation by setting
$\Man{N+1}=\Omega \times (0,\Tf)$, $\Man{1}=\{P_1, P_2,\cdots, P_L\} \times (0,\Tf)$
and $\Man{N}=(\domeg\setminus \{P_1, P_2,\cdots, P_L\}) \times (0,\Tf)$.

\begin{figure}[!h]
    \begin{center}
    \includegraphics[width=0.5\textwidth]{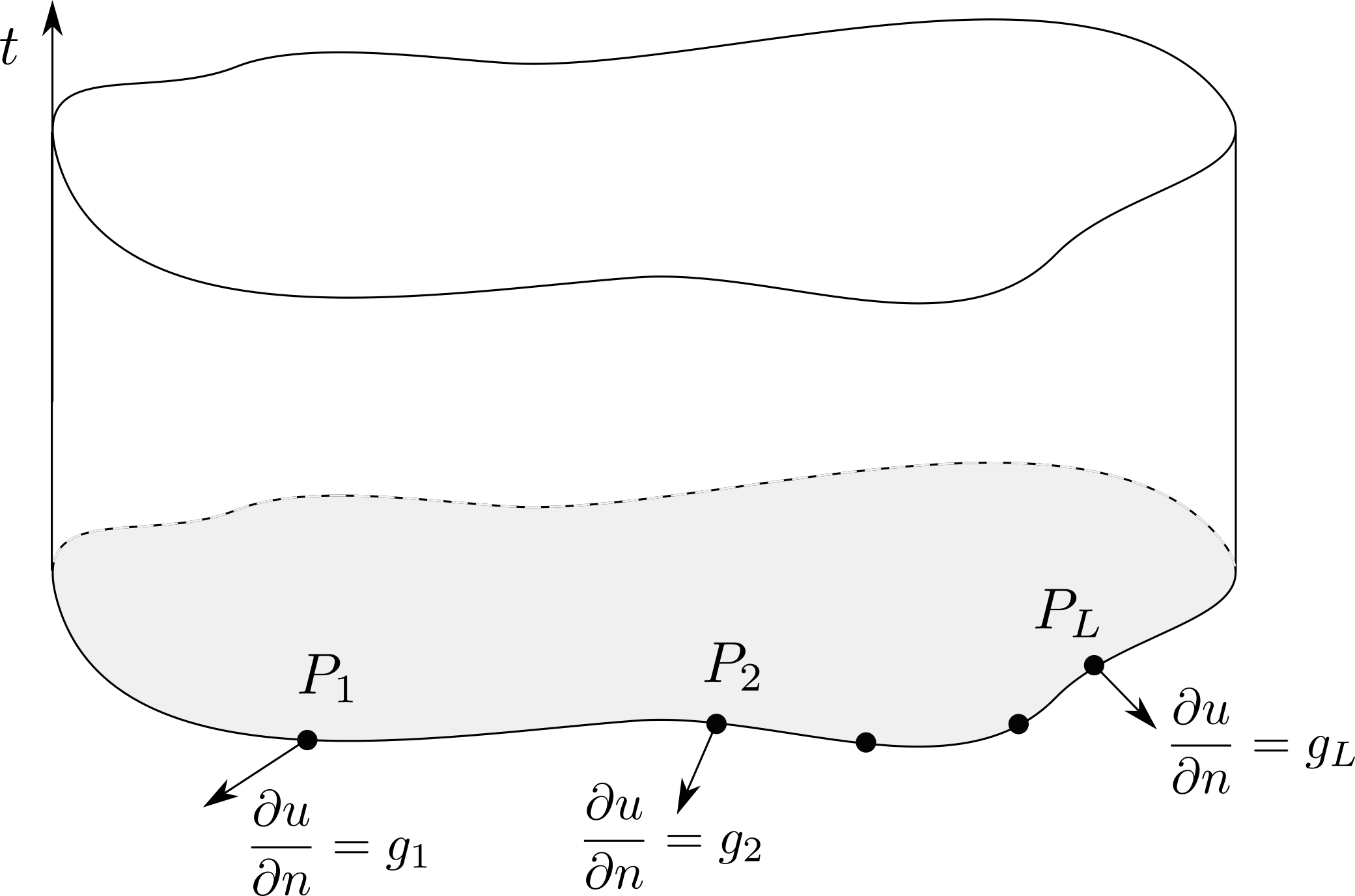}
    \caption{The tanker problem}
    \label{fig:tanker}    
    \end{center}
\end{figure}

The only point is to compute $\F^1$, which is done as in the previous section, except that
we are in $\Man{1}$ and we look for dynamics consisting in staying at $P_i$ for some $i$.

At any $P_i$, we have to consider the convex combinations of $(\gb,\gc,\gl)=((\gb^x,-1),0,\gl) \in\BCL(x,t)$
and $((-n(P_i),0)\;,\,0\;,\,g_i(t))$, \ie 
$$\Big(\mu \gb-(1-\mu)(n(P_i),0)\;,\,0\;,\,\mu \gl+(1-\mu)g(x,t)\Big )$$
for $0\leq \mu\leq 1$ with---since we are on $\Man{1}$---the constraint $\mu \gb^x-(1-\mu)n(P_i)=0$,
leading to
$$\mu \gb^x  = (1-\mu)n(P_i)\; .$$

In order to compute $\F^1(x,t,p_t)$, we have to look at the supremum of 
$ \mu p_t -(\mu \gl+ (1-\mu)g_i(t))$ but taking into account the fact that
$1-\mu= \mu \gb^x\cdot n(P_i)$. Since $\mu$ cannot vanish the condition reduces to
$$u_t + \sup_{\gb^x=\lambda n(P_i),\; \lambda \geq 0}\Big\{- \big(\gl+\gb^x\cdot n(P_i)
    g_i(t)\big)\Big\}\leq 0\quad\hbox{at }P_i\times(0,\Tf)\; .$$

Adequate controllability assumptions yield uniqueness of the stratified
solution for such problem and of course, we can weaken the regularity
assumptions on $\Omega$ which can be a square in $\R^2$ if the corners are
harbors.

\chapter{On the Stability for Singular Boundary Value Problems}
\label{chap:RefSRBC} 

\abstract{The stability properties of stratified solutions in the state-constraints case is not
investigated in its full generality. But relevant examples of the two strategies that can be used
are proposed: $(i)$ the standard half-relaxed limits method, using the standard stability result for
classical viscosity solutions; $(ii)$ borrowing the arguments of the stability results for stratified
solutions.}

\index{Boundary conditions!stability}

We do not plan to try to investigate stability properties of stratified solutions in the
state-constraints case in its full generality here. On one hand, Chapter~\ref{NESR} contains most of
the ideas which are necessary to prove such results. On the other hand however, in order to obtain the
regularity of the limiting subsolutions on the boundary (a necessary ingredient to get a comparison
result allowing to complete the half-relaxed limits method), one has to use, one way or the other,
the ideas of Section~\ref{abl}. But here we face various situations that we think is meaningless  to attempt to
list. 

Let us just indicate that, for example, Lemma~\ref{lem:sufficient.cone} may provide inequalities in
which one can pass to the limit and then use Lemma~\ref{RSub-1} to conclude. We refer the reader to
Chapter~\ref{chap:networks} where such a strategy is used in the case of a generalized network. This example shows that
some extra inequalities---highly depending on the problem at hand---may play a role in the stability
process. But other arguments can also be used.

This variety of situations leads us to present only two examples of strategies that can be used for
passing to the limit in ``singular boundary value problems'', \ie in problems for which the
stratified formulation is necessary to have the right comparison result for the limiting
equation. These two strategies can be described in the following way.

\medskip

\noindent\textbf{Strategy 1} consists in using a standard half-relaxed limits method, with the
standard stability result for classical viscosity solution, \cf Theorem~\ref{hrl}, and then using the
arguments of Chapter~\ref{chap:RefBF} in order to conclude. The advantage of this first approach is
the wide type of perturbations that can be handled for the boundary value problems, the defect being that
it only works in frameworks where classical viscosity solutions and stratified solutions are the
same for the limiting problems.

\noindent\textbf{Strategy 2} relies on borrowing the arguments of the stability results for
stratified solutions. Here the advantages and defects are opposite: we can handle only perturbations
for which a stratified formulation is at hand, but we can a priori address limiting problems where
classical viscosity solutions are not necessarily stratified solutions.

\medskip

The aim of this chapter is to provide two illustrative examples for these two strategies.

\section{Stability via classical stability results}
\index{Vanishing viscosity method!in the stratified framework}

Assume that $u$ is the unique viscosity solution of the Dirichlet problem
\eqref{standardHJB}-\eqref{DP} or the oblique derivative problem \eqref{standardHJB}-\eqref{NP} and
that the conditions of one of the following propositions hold: Proposition~\ref{Dir-nonreg-nW}, \ref{OD-MN}, \ref{OD-gamma-disc},
\ref{prop:DuIs1}, or \ref{OD-gamma-omega-disc} so that classical viscosity subsolutions and
stratified subsolutions are the same. 

We consider approximations by the vanishing viscosity method:
\begin{equation}\label{eqn:approx-e}
    \begin{cases}
        u^\e_t-\kappa(\e)\Delta u^\e+ H_\e (x,u^\e,Du^\e)=0 & \text{in }\Omega_\e \times(0,\Tf)\;,\\
        u^\e(x,0)=u^\e_0(x) & \text{in }\Omega_\e\;.
    \end{cases}
\end{equation}
Here, $\kappa(\e)\geq 0$, $H_\e$ is a continuous function in $\Omegb_\e\times \R \times
\R^N$ and $u^\e_0\in C(\Omegb_\e)$. This problem is associated with either a Dirichlet
boundary condition
\begin{equation}\label{eqn:approx-e-D}
    u^\e(x,t)=\varphi_\e (x) \quad \hbox{on }\Omega_\e \times(0,\Tf)\, ,
\end{equation}
or an oblique derivative boundary condition
\begin{equation}\label{eqn:approx-e-OD}
    \frac{\partial u^\e}{\partial \gamma_\e} =g_\e (x,t) \quad \hbox{on }\domeg_\e\times (0,\Tf)\; .
\end{equation}
In this oblique derivative case, we say that $\tilde \gamma_\e \in \Gamma_\e (x,t)$ for $(x,t)
\in \domeg_\e\times [0,\Tf)$ if there exists a sequence $(x_\delta,t_\delta)\in \domeg_\e\times
[0,\Tf)$ converging to $(x,t)$ such that $\gamma_\e(x_\delta,t_\delta)\to \tilde \gamma_\e$. We 
also use a similar definition for $\e=0$ with $\gamma$ instead of $\gamma_\e$.

Our result is the
\begin{proposition}\emph{--- Stability via classical arguments.}\smsp
    Assume that the above conditions hold and that moreover
    \begin{enumerate}
        \item[$(i)$] $\Omegb_\e\to \Omegb$ and $\domeg_\e \to \domeg$ in the sense of the
        Hausdorff distance;
        \item[$(ii)$] $\kappa(\e)\to 0$, $H_\e\to H$ and $u^\e_0\to u_0$ locally uniformly;
        \item[$(iii)$] according to the Dirichlet or oblique derivative case, 
        \begin{enumerate}
            \item[$(a)$] either $\limssup \varphi_\e = \varphi^*$ and  $\limiinf \varphi_\e =
                \varphi_*$ on $\domeg\times[0,\Tf)$\,;
            \item[$(b)$] or $\Gamma\supset\limsup^*\Gamma_\e$ and $g_\e\to g$ locally uniformly on
                $\domeg\times[0,\Tf)$. 
        \end{enumerate}
    \end{enumerate}
    Then, if $(\ue)_\e$ is a sequence of uniformly bounded viscosity solutions of either the
    Dirichlet ot the oblique derivative problem, 
    $$ u^\e \to u \quad\hbox{locally uniformly in }\Omega \times [0,\Tf)\;,$$
    where of course, $u$ is the solution of either the Dirichlet or oblique derivative problem,
    accordingly.
 \end{proposition}
 
As was announced at the beginning of this chapter, the strategy of proof is clear: the reader will
check easily first that under the assumptions, $\limssup u^\e$ and $\limiinf u^\e$ are classical
viscosity sub and supersolution of either the Dirichlet or oblique derivative problem; then applying
the arguments of Chapter~\ref{chap:RefBF}---in particular the fact that classical viscosity
subsolutions are stratified subsolutions---provides the required comparison result. This allows to
fully apply the half-relaxed limits method in order to conclude.
 
Notice that the approximation by \eqref{eqn:approx-e} is rather general, possibly combining
vanishing viscosity method and an approximation of $H$ by non-convex Hamiltonian. Both perturbations
which cannot be handled by a pure stratified stability result.

\section{Stability via stratified approximations}
 
Let us come back on the tanker problem in $\R^N$, with dimension $N>2$ but on
$\Man{N}=(\domeg\setminus \{P_1, P_2,\cdots, P_L\}) \times (0,\Tf)$, we replace the
state-constraints boundary condition, \ie the condition to ``not unload the cargo outside the
harbors $(P_i)$'': we allow here some smuggling, where one can unload everywhere on the boundary
at a higher cost, the smuggling risk. We model this situation by a Neumann boundary condition on the whole
boundary
$$ \frac{\partial u^\e}{\partial n}=g_\e (x,t) \quad \hbox{on }\domeg_\e\times (0,\Tf)\; ,$$
where $(g_\e)_\e$ is an increasing sequence continuous function such that $g_\e(P_i,t)=g_i(t)$ and
$g_\e (x,t) \to +\infty\ \text{uniformly on any compact subset of }
(\domeg\setminus \{P_1, P_2,\cdots, P_L\})\times  (0,\Tf)\;.$
 
As can be expected, this relaxed boundary condition by allowing smuggling converges to the original
Tanker problem:
\begin{proposition}\emph{--- Stability via stratified arguments, the Taker problem.}\smsp
    Under the assumptions of the tanker problem and the above requirement on $(g_\e)_\e$, the
    unique solution $u^\e \in C(\Omegb \times [0,\Tf))$ of the Neumann problem is uniformly bounded and
    converges to the unique solution $u$ of the tanker problem, uniformly on $\Omegb \times
    [0,\Tf-\delta]$, for any $\delta >0$. 
\end{proposition}
 
\begin{proof} 
    We first remark that $u^\e \leq u$ in $\Omegb \times [0,\Tf)$. This inequality is intuitively
    true since, in terms of control, there are more controls involved in the $u^\e$-problem than in the
    $u$-one. More rigorously, one can show that $\ue$ is a stratified subsolution for the $u$-problem, the
    $\F^1$-inequality resulting from the arguments of Section~\ref{sec:INPofS}. On the other hand,
    for any $\e \leq \bar \e$, $u^\e\geq u^{\bar \e}$, providing a lower estimate so that the limsup
    and liminf of $u_\e$ are well-defined. 
  
    For the convergence, there are a lot of details to be checked but we concentrate on the
    $\F^1$-inequalities on each $P_i$ and leave the rest of the proof to the reader since it is
    based on by now routine arguments---at least we hope so.

    \medskip

    \noindent\textbf{(a)} Let us denote by $\ou=\limssup u^\e$. Let $\bar t > 0$ be a local strict
    maximum point of the function $\ou(P_i,t)-\phi(t)$ where $\phi$ is a $C^1$-smooth. We have to
    prove that
    $$\F^1\big(P_i,\bt,(D_x\phi,D_t,\phi)\big)=\phi_t (\tb)+H^1 (P_i,\tb)\leq 0\; ,$$
    where $$H^1 (P_i,t)=\sup_{b^x=\lambda n(P_i),\; \lambda \geq 0}
    \Big\{- \big(l+b^x\cdot n(P_i) g_i(t)\big)\Big\}.$$
    
    In order to do so, we consider the function
    $$ (x,t) \mapsto u^\e (x,t)-\phi(t)-C|x-P_i| - C^{1/2}(t-\tb)^2\; ,$$
    for some large constant $C>0$ to be chosen later on.

    We look at this function in a small compact neighborhood $\VV$ of $(P_i,\tb)$. It is clear that this
    function achieves its maximum in $\mathcal{V}$ at some point $(\xe,\te)$ and, by choosing $C$
    large enough, we make sure that $(\xe,\te)$ cannot be on the boundary of $\mathcal{V}$; indeed
    $u^\e$ and $\phi$ being uniformly bounded, there exists $M>0$ such that
    $$ C|\xe-P_i| + C^{1/2}(\te-\tb)^2\leq M$$
    so that $(\xe,\te)$ is close to $(P_i,\tb)$ for $C$ big enough.

    \medskip

    \noindent\textbf{(b)} Next we have to investigate several cases:\\[2mm]
    1. The case $\xe \in \Omega$ cannot happen because of the normal controllability assumption. 
    Indeed, if $e=(\xe -P_i)|\xe-P_i|^{-1}$ the $H$-inequality
    $$ \phi_t(\te)+2C^{1/2}(\te-\tb)+ H(\xe,\te,Ce)\leq 0$$
    cannot hold for $C$ large enough. Notice that the size of $C$ to rule out this case is
    independent of $\e$.\\[2mm]
    2. If $\xe\in \domeg \setminus \{P_i\}$, the $\F^N$-inequality reads
    $$\sup_{(\theta b^x-(1-\theta)n)\cdot n=0}
    \Big\{\theta\big(\phi_t(\te)+ 2C^{1/2}(\te-\tb)\big)-\big(\theta b^x-(1-\theta)n\big)\cdot Ce -
    \big(\theta l +(1-\theta)g_\e(\xe,\te)\big)\Big\}\; .$$
    But since the distance function $d(\cdot)$ to the boundary is $C^{1,1}$, we have
    $$ 0=d(\xe)-d(P_i)=-(\xe-P_i)\cdot n(\xe)+o(|\xe-P_i|)\; ,$$
    where the ``$o$'' is uniform in $\e$.
    Hence $e \cdot n(\xe)=o(1)$, \ie these two unit vectors are almost orthogonal. 
    On the other hand, by the local normal controllability, there
    exists $b^x \in T_{\xe}\domeg$ such that $b^x \cdot e \geq \eta>0$ for some fixed $\eta$. 
    We use this $b^x$ with $\theta=1$ in the $\F^N$-inequality which implies
    $$
        \phi_t(\te)+2C^{1/2}(\te-\tb)-C b^x\cdot e-l \leq 0\; .
    $$
    Again this case cannot happen if $C$ is chosen large enough.\\[2mm]
    3. The only remaining case is $\xe =P_i$ and the strict maximum point property for
    $\ou(P_i,t)-\phi(t)$, which is even more true for $\ou(P_i,t)-\phi(t)- C^{1/2}(t-\tb)^2$,
    implies that $\te \to \tb$. 

    Now, the $\F^N$-inequality at $(P_i,\te)$ implies the $\F^1$-one since, by definition, $\F^1\leq
    \F^N$ in our case. It finally remains to let $\e\to 0$, leading to the result.

    \medskip

    \noindent\textbf{(c)} We want to comment on two other points, in order to complete the proof.
    
    \noindent -- The case $\tb=0$ can be treated similarly but it leads to an inequality of the
    form $\min(u-u_0,\F^1)\leq 0$ at $(P_i,0)$. Indeed, in the proof we may face the case $\xe \in
    \Omega$ and $\te=0$ for which we have $u^\e(\xe,\te)\leq u_0(\xe)$. This inequality at $(P_i,0)$
    allows to show that $u(P_i,0)\leq u_0(P_i)$ by the methods of Section~\ref{IC-HJB} and therefore
    gives an information leading to show that \HBAIDCP holds.

    \noindent -- The regularity of $b=b(x,\alpha)$ together with the normal controllability
    assumption allows to use Lemma~\ref{lem:sufficient.cone} and then Lemma~\ref{RSub-1} is a very
    easy way, both for $u$ or the $(u^\e)$: indeed, if $(x,t)\in \domeg \times (0,\Tf)$, the normal
    controllability implies that there exists $\bar \alpha$ such that $b(x, t, \bar \alpha)\cdot
    n(x)<0$. On one hand, this property clearly gives \eqref{cone}, and on the other hand, by
    Proposition~\ref{sub-up-to-b}, we have
    $$ w_t - b(x, t, \bar \alpha)\cdot D_x w\leq l(x,t,\bar \alpha) \quad \hbox{on }
    [\Omegb \times (0,\Tf)] \cap [B(x,r)\times (t-r,t+r)]\; ,$$
    for $r$ small enough and for $w=u,u^\e$. Of course, we can use a standard stability result to
    pass to the limit in this inequality and this gives the regularity of $\ou$ on $\domeg \times
    (0,\Tf)$.

    With this two additional points, the proof is complete.
 \end{proof}
 
 \begin{remark}
     As the reader may have noticed, the above proof is nothing but a (slightly)
     simplified version of the proof of the stability result for stratified subsolutions. It is
     rather strange that there is no real penalization of the distance to $\Man{1}$ since $C$ is
     chosen large, but never tends to infinity. On the contrary, the normal controllability
     assumption forces the maximum points to be on $\Man{1}$, showing how important this assumption
     is.
 \end{remark}

This section is concerned with some stability results involving Dirichlet and 
Neuman problems in presence of singularities. The aim here is to mix results
from Sections~\ref{NESR} and \ref{chap:RefBF}. However, we do not try to cover
all the possible results here since such formulations would imply a lot of
technicalities and make everything very difficult to read.

We prefer instead to present a few illustrative examples that highlight the
important ideas and results.

So, the main theme of this section is the following: we start from standard, continuous
boundary condition problems and we want to pass to the limit when singularities are created 
(in the limit) at the boundary.

\section{A concrete application: singularities in Dirichlet problems}
\index{Dirichlet problems!stability}

Starting from a sequence of smooth domain $(\Omega_\e)_{\e>0}$, and smooth 
(say continuous) boundary data $(\varphi_\e)_{\e>0}$, there are essentially 
two mechanisms that can yield some singularities as $\e\to0$:
\begin{enumerate}
    \item[(a)] the boundaries $\partial\Omega_\e$ converge to a boundary 
        $\partial\Omega$ presenting a non-trivial stratification, 
        ie. $\partial\Omega$ is not a $C^1$ hypersurface of $\R^N$;
    \item[(b)] the sequence $\varphi_\e$ converges to a discontinuous 
        limit boundary condition $\varphi$.
\end{enumerate}
Of course, a general problem may involve both mechanisms and even at the same
location in the limit.  But here we are going to separate both cases in order
to keep things as simple as possible. The good news being that both situations
can be handled by the stratified framework.

\subsection{Non-smooth domains}

We consider here the case of a square $\Omega:=(0,1)^2\subset\R^2$, already considered in 
Example~\ref{ex:square} and a stratified problem
\begin{equation}\label{pb:square}
    \begin{cases}
        u_t+H(x,t,Du)=0 & \text{in }\Omega\times(0,\Tf)\;,\\
        u(x,t)=\varphi(x,t) & \text{on }\partial\Omega\times(0,\Tf)\;,\\
        u(x,0)=u_0(x) & \text{in }\Omega\;.
    \end{cases}
\end{equation}
where $\varphi$ and $u_0$ are W-adapted to the natural stratification of
$\partial\Omega\times(0,\Tf)$. For simplicity here, let us even assume that
$\varphi,u_0$ are continuous and satisfy the compatibility condition
$\varphi(x,0)=u_0(x)$ for $x\in\partial\Omega$ and that $H$ is as in Chapter~\ref{chap:RefBF}, with the right
\NCBCL-\TCBCL which are adapted to the stratification of the boundary.

Here the problem is singular due to the geometric nature of $\Omega$ which is
clearly only a Lipschitz domain. As we have already seen, this singular
Dirichlet problem can be handled in the stratified framework under some quite
general hypotheses. We refer to the corresponding section for all the details.
Let us just mention here that there is an underlying set-valued map $\BCL$ defined on
$\bar\Omega\times[0,\Tf]$, taking into account the Dirichlet boundary data
$\varphi$ on $\partial\Omega\times(0,\Tf)$ as well as the initial data $u_0$ on
$\Omega\times\{0\}$, which allows to get a unique stratified solution
of~\eqref{pb:square}. 

Of course, the stratification $\M$ of $\Omega\times[0,\Tf]$ is time-independent and given by
$\M^3:=\Omega\times \R $, $\M^1:=\{\{P_i\}\times \R ;i=1..4\}$ where $P_i$ are the four corners of
the square and $\M^2=\{\E_i \times \R; i=1..4\}$, where the $(\E_i)_{i=1..4}$ are the four (open)
sides of the square.

The question is to know whether can we approximate the stratified solution by a sequence of
```standard'' problems, defined in smooth domains. Or equivalently, can we
identify the limit of such regular problems when the geometry of the boundary
yields some singularities in the limit ?

In order to answer the question, we approximate the square by a sequence of
smooth domains $(\Omega_\eps)_\e$ ``converging'' to $\Omega$. 
For simplicity, we choose here
$$\Omega_\eps:=\{x\in\Omega:{\rm dist}(x,\partial\Omega)>\eps\}\;.$$
This specific sequence has some advantages like preserving the symmetries of the problem, 
being convex, monotone, included in $\Omega$, see
Fig.~\ref{fig:approx-stab}. However, we could use more general approximations
provided they satisfy some properties, the most important one being of course
the convergence to $\Omega$ in the sense of Definition~\ref{ncos}.

\begin{figure}[htp]
\begin{center}
    \includegraphics[width=0.5\textwidth]{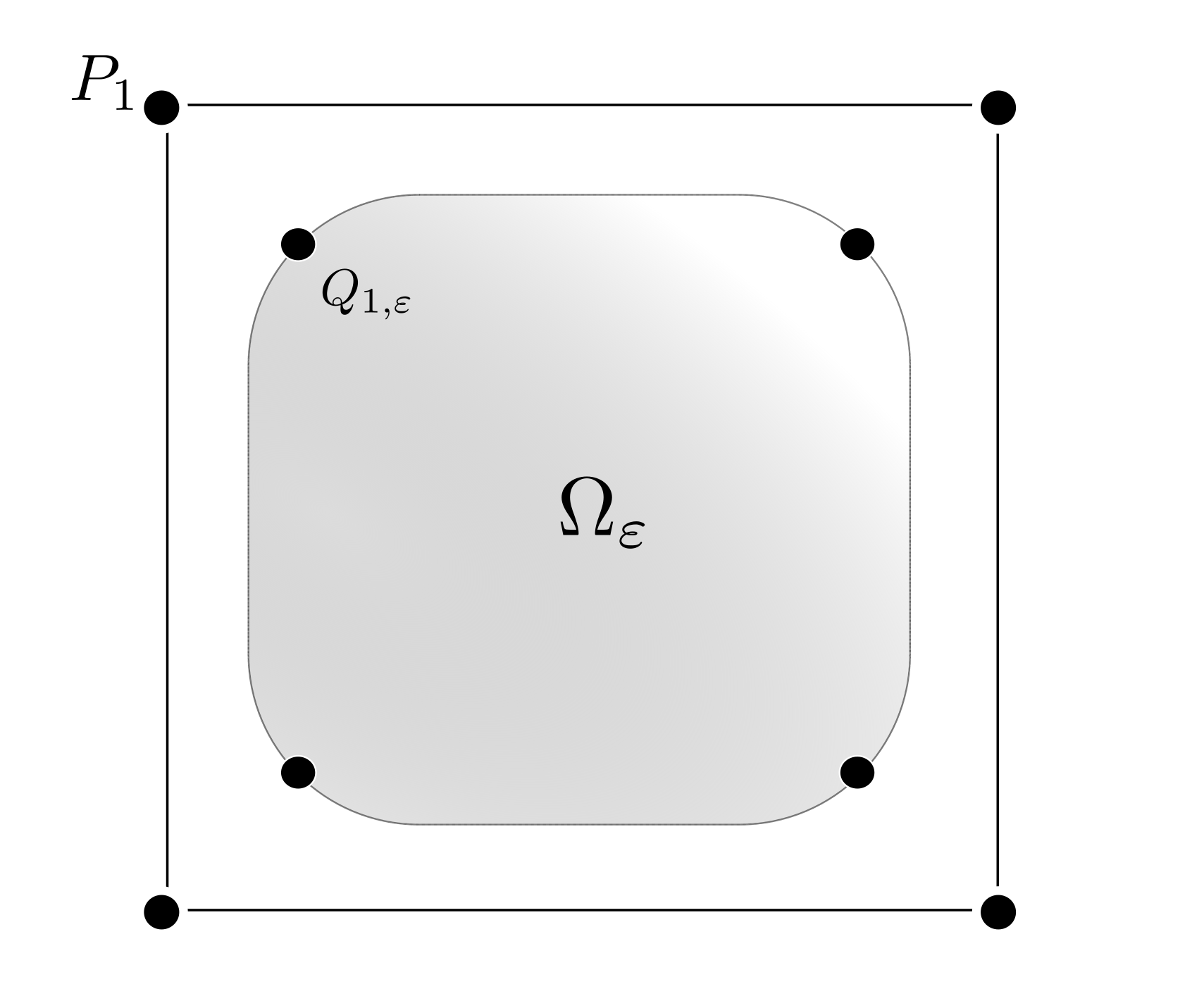}
    \caption{Approximation of the square}
    \label{fig:approx-stab}
\end{center}
\end{figure}

Concerning the Dirichlet boundary data, we introduce a sequence of 
continuous function $\varphi_\e$ defined on $\partial\Omega_\e\times(0,\Tf)$,
satisfying the following convergence property: for any sequence of points
$(x_\e,t_\e)\in\partial\Omega_\e\times(0,\Tf)$
converging to a point $(x,t)\in\partial\Omega\times(0,\Tf)$, 
$$\lim_{\e\to0}\varphi_\e(x_\e,t_\e)=\varphi(x,t)\;.$$
Moreover, we assume that for each $\e>0$, the compatibility condition 
$\varphi_\e(x,0)=u_0(x)$ is valid on $\Omega_\e$. Under these assumptions, we know that
for any $\e>0$ there is a unique viscosity solution $u_\e$ of the following problem
 $$\begin{cases}
     (u_\e)_t+H(x,t,Du_\e)=0 & \text{in }\Omega_\e\times(0,\Tf)\;,\\
        u_\e(x,t)=\varphi_\e(x,t) & \text{on }\partial\Omega_\e\times(0,\Tf)\;,\\
        u_\e(x,0)=u_0(x) & \text{in }\Omega_\e\;.
    \end{cases}$$
The main result on the convergence is the
\begin{theorem}\label{thm:stab.disc.geo}\emph{--- Stability for a singular domain.}\smsp 
    As $\e\to0$, the sequence $(u_\e)_\e$ converges locally uniformly in 
    $\Omega\times(0,\Tf)$ to the stratified solution $u$ of problem \eqref{pb:square}.
\end{theorem}

\begin{proof}
    The proof consists just in assembling various results that appear in the previous 
    parts of this book. However, there are a few things to do before that.

    \noindent\textsc{A. Stratifying the problem —}
    Thanks to our assumption on $\Omega_\e$, we can easily define four points 
    $Q_{i,\eps}\in\Omega_\e$, located on the square diagonals.
    Notice that $Q_{i,\eps}$ actually minimizes the distance to $P_i$ from $\Omega_\e$, 
    so that $Q_{i,\e}\to P_i$ for each $i=1..4$ as $\e\to0$.
    Those points allow us to define a super-stratification $\M_\e$ of
    $\Omega_\e\times\R$ where $\Man{1}_\e=\{Q_{i,\e}\times\R;i=1..4\}$
    and the four others parts of $\partial\Omega_\e\times\R$ 
    are elements of $\Man{2}_\e$. 

    \noindent\textsc{B. Defining the sets $\BCL_\e$ and the Hamiltonians —} 
    here we follow exactly 
    Section~\ref{RefBF:Dir}: $\BCL_\e$ is just a modification of $\BCL$
    on $\partial\Omega_\e\times(0,\Tf)$, where we extend it as the 
    convex enveloppe of 
    $$\BCL(x,t)\,\cup\,\{(0,1,\varphi_\e(x,t))\}\;.$$
    Of course this new $\BCL_\e$ is $\M_\e$-adapted. 
    
    Using Proposition~\ref{ajout}, we introduce three Hamiltonians:
    first, $\mathbb{F}^3_\e(x,t,u,p):=p_t+H(x,t,p_x)$ defined in
    $\Omega_\e\times(0,\Tf)$, which is just related to the equation inside
    $\Omega_\e$.
    Then, on each point $Q_{i,\e}$ of $\Man{1}_\e$, we set 
    $$\mathbb{F}_\e^1(x,t,u,p):=\max(u(x,t)-\varphi_\e(x,t), u_t + H^1(Q_{i,\e},t)) \;,$$
    where 
    $$ H^1(Q_{i,\e},t):= \max_{b(Q_{i,\e},t,\alpha)=0}\,(-l(Q_{i,\e},t,\alpha))\; ,$$
    and similarly, we set, on $\E_i \times (0,\Tf)$
    $$\mathbb{F}^2_\e(x,t,u,p)=\max(u(x,t)-\varphi_\e(x,t), u_t + H^2(x,t,p)) \;,$$
 where 
    $$ H^2(Q_{i,\e},t,p):= \max_{b(x,t,\alpha)\in T_{x}E_i}\,(-b(x,t,\alpha)\cdot p-l(Q_{i,\e},t,\alpha))\; ,$$
   
     when
    $x$ belongs to $\partial\Omega_\e\setminus\{Q_{i,\e};i=1..4\}$.
    
    In this setting,  $u_\e$ can be seen as a the stratified solution of
    $\mathbb{F}_\e(x,t,u,D_xu,u_t)=0$, meaning that for each $k=1..3$, 
    a $\mathbb{F}_\e^k$-subsolution inequality holds on $\Man{k}_\e$ 
    while $u_\e$ is a (classical viscosity) $\mathbb{F}^3$-supersolution. 

    \noindent\textsc{C. Passing to the limit —}
    In the sense of Definition~\ref{ncos}, $\M_\e$ converges to 
    $\M$, the  stratification of $\partial\Omega$.

    In order to apply Theorem~\ref{nsr}, we just notice that by construction,
    $\BCL_\e$ converges to $\BCL$ in the sense of Lemma~\ref{lem:stab.bl.ham} 
    (because of our assumption on the convergence of $\varphi_\e$). This allows us
    to use the theorem, and conclude as in Corollary~\ref{cor:stability}.
\end{proof}

\subsection{Non-smooth data}

Here we assume that $\Omega$ is a fixed smooth domain in $\R^2$ and that we have
a sequence of boundary data $\varphi_\e$ continuous on $\partial\Omega\times[0,\Tf]$, 
converging to some $\varphi$ which may be discontinuous at some isolated 
points of $\partial\Omega\times[0,\Tf]$. For simplicity here, we assume that there is only one
point $P_0=x_0\in \domeg$ with $t_0>0$ at which $\varphi$ is not continuous. Hence, more precisely we have
$$ \limssup \varphi_\e = \limiinf \varphi_\e = \varphi \quad \hbox{on  }[\domeg \setminus P_0]\times [0,\Tf)\; ,$$
and there exists a sequence $x_\e \to x_0$ such that
$$\varphi_*(x_0,t):=\lim \varphi_\e(x_\e,t)\quad \hbox{for any }t>0\;.$$

We assume that we are in the framework of Chapter~\ref{chap:RefBF}, \ie
$$\BCL(x,t):= \{(b(x,t,\alpha),0,l(x,t,\alpha));\ \alpha \in A\}\quad \hbox{for  }(x,t) \in \Omegb \times [0,\Tf]\; .$$ 
At the $\e$-level we set $\BCL_\e(x,t)=\BCL(x,t)$ if $(x,t)\in \Omega\times(0,\Tf)$ and 
$$\BCL_\e (x,t)=\overline{\mathrm{co}}\Big(\BCL(x,t)\,\cup
\,(0,1,\varphi_\e(x,t)  \Big)\quad \hbox{if  }x\in \domeg\times(0,\Tf)\;.$$
Of course, for $t=0$ we have to incorporate the term corresponding to the initial data.
For $\e=0$, the only change is that we have to replace $\varphi_\e(x,t)$ by $\varphi_*(x,t)$.

From the pde viewpoint, we are considering the following boundary value problem
$$\begin{cases}
     (u_\e)_t+H(x,t,Du_\e)=0 & \text{in }\Omega\times(0,\Tf)\;,\\
        u_\e(x,t)=\varphi_\e(x,t) & \text{on }\partial\Omega\times(0,\Tf)\;,\\
        u_\e(x,0)=u_0(x) & \text{in }\Omega\,
\end{cases}$$
which has a unique viscosity solution since $\varphi_\e$ is continuous, 
and the singular limit problem
$$\begin{cases}
     u_t+H(x,t,Du)=0 & \text{in }\Omega\times(0,\Tf)\;,\\
        u(x,t)=\varphi_*(x,t) & \text{on }\partial\Omega\times(0,\Tf)\;,\\
        u(x,0)=u_0(x) & \text{in }\Omega\;.
\end{cases}$$
Since $\varphi$ is discontinuous at $P_0$, we have to consider a stratification 
of $\partial\Omega\times(0,\Tf)$ involving $\M^1:=\{\{P_0\}\times(0,\Tf)\}$.
The result is the following
\begin{theorem}\emph{--- Stability for a singular limit data.}\smsp
    As $\e\to0$, the sequence of viscosity solutions $(u_\e)_\e$ converges to the
    stratified solution of the limit problem, associated with the boundary data
    $\varphi_*$.
\end{theorem}

\begin{proof}
    The strategy is essentially the same as for the case of a singular boundary: at the $\e$-level,
    we introduce a super-stratification $\M_\e$ with $\Man{1}_\e:=\{x_\e\}\times \R$ which converges to 
   $\Man{1}_\e:=\{x_0\}\times \R$ in order to take into account the discontinuity
    at $P_0$. 

    With this new stratification, $(u_\e)_\e$ can be seen as the stratified solution of the Dirichlet problem with the
    boundary condition $\varphi_\e$, meaning that the $\F^k_\e$-subsolution inequalities ($k=1..3)$
    coupled with a $\F^3_\e$-supersolution inequality.

    We do not detail every Hamiltonian since they are similar to those that we
    used in Theorem~\ref{thm:stab.disc.geo}. Notice that we only get $\limsup^*\BCL_\e(x,t) \supset \BCL_*(x,t)$ because the limsup contains
    any possible relaxed limit of $\varphi_\e$, not only $\varphi_*$.

    Using the convergence results of Theorem~\ref{nsr} and Corollary~\ref{cor:stability}
    we see that $u_\e$ converges locally uniformly to the stratified solution $u$ of the Dirichlet problem 
    associated to the boundary data $\varphi_*$, which we wanted to prove.
\end{proof}

\chapter{Further Discussions and Open Problems}
\label{chap:openpb-partV}

\abstract{In this chapter, the main comments concern the ``regularity of subsolutions'' on the
boundary; probably a lot of work remains to be done to have completely satisfying results.  For the
extensions, stationary problems, more general dependence in time and unbounded control problems are
considered.}

In this part, we have extended the results of Part~\ref{stratRN} without---apparently---much more
additional difficulties.  This sentence is at the same time true since no new arguments is really
needed in the state-constraints framework but partly wrong since the two difficulties we describe in
the introduction of this part are rather serious.

This gives us the opportunity to make an overall assessment of our approach to stratified problems:
our point of view was to look for a general framework for which we could prove ``nice results'' for
both the HJB Equation (comparison, stability, etc.) and the associated control problem (continuity
of the value function in particular). We end up thinking that our two main assumptions, namely the
tangential continuity and the normal controllability, are indeed playing a key role. Even if,
obviously, this framework does not cover all the possible interesting cases, we keep thinking that it
is a ``natural'' general setting to treat problems with discontinuities. And actually we hope to
have convinced the reader that these two main assumptions are useful for proving {\em any result}.

But, in order to use these assumptions, the important notion of ``regular subsolution'' appears
everywhere as the reader can check it through the table of contents and actually, in
Part~\ref{stratRN}, this regularity was ensured by the normal controllability.

Obviously the new point in Part~\ref{S-BC} is that we have a boundary and obtaining the regularity
of subsolutions on the boundary is far more delicate. In the classical theory, this difficulty was
already appearing in the Dirichlet problem: we recall the idea of a ``cone condition'' for
state-constrained problem which was initiated by Soner \cite{Son1,Son2} and used in different ways
for example in \cite{BP2,BP3}.

We have tried to provide ideas in order to prove this regularity of subsolutions on the boundary.
But the methods we use to turn around the difficulties are far from being as general as one could
hope and, probably, they have to be improved specifically on a per-example basis. As we pointed out
in the section related to the regularity of subsolutions on the boundary, the stratified approach
can handle such a wide variety of non-smooth domains that providing a general theory to treat all of
the situations is hopeless.\\

\

As in Part~\ref{stratRN}, several extensions can be considered, for instance:
stationary problems, more general dependence in time, unbounded control problems. Let us make
several comments on these points.

Concerning our study of classical boundary value problems, a puzzling question is related to the
optimality of the results we obtain: from a technical point of view, it is not clear that the
various conditions we use to show that Ishii's subsolutions are stratified subsolutions are really
necessary. But the main question is, of course, whether it is possible to obtain comparison results
between Ishii sub and supersolutions in a more general setting or if these conditions are more or
less necessary. Unfortunately we do not have any example or counter-example for these questions.

\medskip

\noindent $(i)$ To begin with, stationary problems with boundary data should not pose any major difficulty
apart from those already dealt with in this book. Actually, the reader may get direct translations
of our results by assuming that the $\BCL$-set that we construct is independent of $t$, as well as
the sub and supersolutions.

\medskip

\noindent $(ii)$ Another problem concerns the generality of the equations we handle: we have chosen to
consider only the case of HJB Equations in the standard form $u_t+H(x,t,D_xu)$ but one may wonder
what can be done about equations involing ``gradient constraint'', typically
$$
\max(u_t + H(x,t,D_x u),|D_x u|-1)=0 \quad \hbox{in $ \Omega \times (0,\Tf)$} \; .
$$
Such feature would modify the way the initial data is taken into account, but on the other hand we
get Lipschitz continuous subsolutions for free due to the constraint.

\medskip

\noindent $(iii)$ The treatment of boundary value problems that we provide ``for the sake of simplicity'' in
bounded domains and for a time-independent stratification of the boundary (and with a standard
equation inside the domain) can certainly be extended to the case of unbounded domains with a
time-dependent stratification of the boundary and with an equation with discontinuities inside the
domain. 

As we already mentioned it above, if the stratification inside the domain does not interfere with
the one of the boundary, such extension is easy. If there is an interference, certainly the basic
arguments of the comparison proof are not affected but one has to check that the regularity of
subsolutions on the boundary is true. We leave these checkings to the reader.

\medskip

\noindent $(iv)$ Time-dependent stratifications of the boundary should not cause a major difficulty
in bounded domains but we did not check it precisely.  It is clear that proving that an Ishii
subsolution is a stratified subsolution is a local proof and neither time-dependent stratifications
of the boundary nor unbounded domains seem so difficult to handle. 

There is anyway a difficulty which appears in
unbounded domains and which may be even worse in case of time-dependent stratifications of the
boundary; we describe it now for oblique derivative problem but it arises in any type of boundary
conditions, with different forms.

Let us consider the oblique derivative problem as an example, and even in the simplest case of a smooth boundary.
We have for $\F^N$:
$$\F^N(x,t,(p_x,p_t))= \sup \Big\{\theta p_t -\big(\theta \gb^x-
    (1-\theta)\gamma\big)\cdot p_x -\big(\theta \gl +(1-\theta)g\big)\Big\}\;\hbox{on }\Man{N}\; ,$$
    where the supremum is taken on all $(\gb,0,\gl) \in \BCL (x,t)$ such that there exists $\theta \in
    (0,1)$ satisfying $(\theta \gb^x-(1-\theta)\gamma)\cdot n(x)=0$, where $n(x)$ is the unit outward
    normal to $\domeg$ at $x$.
    
We have made the following remark above:\\

\noindent {\em ``The $\theta$'s which satisfy $(\theta \gb^x-(1-\theta)\gamma)\cdot n(x)=0$ for some
$x,t,\gb^x$ are  bounded away from $0$ because of Assumption~\ref{assump:NP}: indeed $$ \theta
(\gb^x+\gamma)\cdot n(x)\geq \nu >0\; .$$ This property, and the analogous ones for the $\F^k$'s,
will play a key role in checking Assumption~\LOCaEV.''}\\

Indeed, if all these $\theta$ are greater than $\bar \theta >0$ then the dependence of
$\F^N(x,t,(p_x,p_t))$ in $p_t$ allows to check Assumption~\LOCaEV as in Section~\ref{sect:htc} and
this is true if Assumption~\ref{assump:NP} strictly holds with a fixed $\nu$. But if $\nu$ depends
on $r$ if $x \in B(0,r)$ then the ``$\bar \theta$'' may also depend on $r$ and the checking of
\LOCaEV may become a problem.

As we already mentioned it above, such difficulty may arise for all type of boundary conditions
(Dirichlet, Neumann or mixed) under different form where the time-dependence of the stratifications
of the boundary can play a role since the various parameters $\theta$ depend on $T_{(x,t)}\Man{k}$.
As we remark it above, showing that an Ishii's subsolution is a stratified subsolution is done
through a local proof and therefore is not affected by the boundedness or unboundedness of the
domain but then it remains to check that the obtained stratified problem satisfies the right
assumption for the comparison result.


\part{Investigating Other Applications}
\label{part:compl-appl}
\fancyhead[CO]{HJ-Equations with Discontinuities: Investigating Other Applications}


\chapter{KPP-Type Problems with Discontinuities}
\label{chap:KPP}

\index{KPP type problems}

\abstract{KPP-Type Problems are usual applications for viscosity solutions theory since their
complete treatment requires a combination of (i) stability result to pass to the limit in a vanishing
viscosity-like context, (ii) comparison result and (iii) connection with optimal control. A perfect playground
to test the progress in the discontinuous framework. Several results combining the results of all
the previous parts are exposed.}

\section{Introduction on KPP Equations and front propagations}

In this chapter, we are interested in Kolmogorov-Petrovsky-Piskunov \cite{KPP} type 
equations (KPP in short), whose simplest form is
\begin{equation}\label{eq:kpp.simple}
    u_t -\frac12 \Delta u = cu(1-u)\quad \hbox{in  }\R^N\times (0,+\infty)\; ,
\end{equation}
where $c$ is a nonnegative constant.

Such reaction-diffusion equation appears in several different models in Physics
(combustion for example) and Biology (typically for the evolution of population) 
and, in all these applications, one of the main interest comes from the large time 
behavior of the solutions which is mainly described in terms of front propagations.
One of the main ingredients to understand this behavior is the study of the
existence of {\em travelling waves solutions}, \ie solutions which can be
written as $$ u(x,t):= q(x\cdot e - \alpha t)\; ,$$
where $q:\R\to [0,1]$ is a smooth enough function, 
$e \in \R^N$ is such that $|e|=1$, and $\alpha \in \R$. 
The travelling wave connects the instable 
equilibrium $u=0=q(-\infty)$ with the stable one $u=1=q(+\infty)$.

The connection between these 
travelling waves solutions and front propagation phenomenon is clear: the
existence of such a solution implies that hyperplanes $x \cdot e = constant$
propagate with a normal velocity $\alpha$. And clearly, understanding the 
propagation of such flat fronts is a key step towards dealing with 
more complicated fronts.

The case of KPP Equations is complicated in terms of travelling waves: while for 
other nonlinearities---for example cubic non linearities like $f(u)=(u-\mu)(1-u^2)$---
there exists a unique characteristic velocity, KPP Equations admit a critical 
velocity $\alpha^*>0$ such that travelling waves solutions exist for all $\alpha \geq \alpha_*$.
And it is well-known that the large time behavior of the solutions (in particular the choice of
the velocity) depends on the behavior at infinity of the initial data. Actually this large time
behavior can be rather complicated since it can be explained by the ``mixing'' of several different
travelling waves as explained in 
Hamel and Nadirashvili \cite{HN}.

We are going to concentrate here on the case where the minimal velocity 
$\alpha_*$ is selected. In this case it is known that $\alpha_*=\sqrt{2c}$ and
that the large time behavior of the solutions of the KPP Equation is described by a
front propagating with a $\sqrt{2c}$ normal velocity, where the front separates
the regions where $u$ is close to $0$ and to $1$.

In order to prove this result, Freidlin \cite{F} introduced a scaling in space
and time $\displaystyle (x,t)\to (\frac{x}{\e},\frac{t}{\e})$ which has the
double advantage to preserve the velocities and to allow to observe in finite
times the large time behavior of the solution by examining the behavior of the
scaled solution as $\e\to 0$. Hence one has to
study the behavior when $\e \to 0$ of 
$$ \ue(x,t)=u \left(\frac{x}{\e},\frac{t}{\e}\right)\; ,$$
which solves the singular perturbation problem
$$ (\ue)_t -\frac{\e}2 \Delta \ue = \frac{c}{\e}\ue(1-\ue)
\quad \hbox{in  }\R^N\times (0,+\infty)\; .$$
We complement this pde with the initial data
$$ \ue(x,0)=g(x) \quad \hbox{in  }\R^N\; ,$$
where $g : \R^N \to \R$ is a compactly supported continuous function satisfying
$0 \leq g(x) \leq 1$ in $\R^N$. 

The reader might be surprised by this unscaled initial data but, in this
approach, the role of $g$ is just to initialize the position of the front,
given here by the boundary of the support of $g$,
$\Ga_0:=\partial\,\mathrm{supp}(g)$.

In this context, the following properties can be proved:
$$ \ue(x,t) = \exp\left( - \frac{I(x,t) + o(1)}{\e} \right)\; ,$$
where $I$ is the unique viscosity solution of the variational inequality 
$$ \min \left( I_t + \frac 12 |DI|^2+ c ,I \right) = 0 \quad\hbox{in 
}\R^N \times (0,+\infty)\; , $$
with
$$ I (x,0) = \begin{cases}
0 & \hbox{if  }x\in \Ga_0 \\ +\infty & \hbox{if  }x\in \R^N 
\backslash \Ga_0.\end{cases} $$  
 Moreover $I=\max(J,0)$ where $J$ is the unique 
viscosity solution of
$$ J_t + \frac 12 |D J|^2+ c = 0 \quad\hbox{in 
}\R^N \times (0,+\infty)\; .$$

The importance of this second part of the result is to allow for an easy
computation of $J$, and therefore $I$, through the Oleinik-Lax formula 
$$J(x,t) = \frac{[d(x,\Ga_0)]^2}{t} -ct \; ,$$
where $\Ga_0=\mathrm{supp}(g)$. Hence $\ue(x,t) \to 0$ in the domain $\{I>0\}=\{J>0\}=
\{d(x,\Ga_0) >\sqrt{2c}t\}$ and it can be shown that $\ue(x,t) \to 1$ in the interior of 
the set $\{I=0\}=\{J\leq 0\}=\{d(x,\Ga_0) \leq \sqrt{2c}\,t\}$.
Therefore the propagating 
front is $\Gamma_t = \{d(x,\Ga_0) = \sqrt{2c\,}t\}$ which means a propagation with 
normal velocity $\sqrt{2c}$ as predicted by the travelling waves.

Such kind of results, in the more general cases of $x,t$ dependent velocities
$c(x,t)$, diffusion and drift terms, were obtained by Freidlin \cite{F} using
probabilistic Large Deviation type methods and later, pde proofs, based on
viscosity solutions' arguments, were introduced by Evans and Souganidis \cite{ES1,ES2}.
They were then developed not only for KPP Equations but for other reaction-diffusion
equations by Barles, Evans and Souganidis \cite{BES}, Barles, Bronsard and
Souganidis \cite{BBrS}, Barles, Georgelin and Souganidis \cite{BGS}. Later, these front
propagation problems were considered in connections with the ``levet-set
approach'': one of the first articles in this direction was the one by Evans,
Soner and Souganidis \cite{ESS} (see also Barles, Soner and Souganidis \cite{BSS}).
The most general results in this direction are obtained through the
``geometrical approach'' of Barles and Souganidis \cite{BS-nga}. A complete overview
of all these developments can be found in the CIME course of Souganidis \cite{S-CIME}
where a more complete list of references is given.

Of course, the aim of this section is to extend the results for KPP Equations to the 
case of discontinuous diffusions, drifts and reaction terms. But before doing so, we come back to
the main steps of the above mentioned result: one has to
\begin{enumerate}
\item introduce the change of variable $I_\e:= -\e \log(\ue)$ and show that $I_\e$ 
is uniformly locally bounded;
\item pass to the limit by using the half-relaxed limits method in the equation
    satisfied by~$I_\e$;
\item prove a strong comparison result for the variational inequality which allows 
to prove that $I_\e \to I$ locally uniformly in $\R^N \times (0,+\infty)$;
\item show that $I=\max(J,0)$, when this is true---see just below. 
\end{enumerate}
All these steps are classical, except perhaps the last one which is related to the 
{\em Freidlin condition}: $J$ is given by a formula of representation given by the 
associated control problem and Freidlin's condition holds if the optimal 
trajectories for points $(x,t)$ such that $J(x,t)>0$ remain in the domain $\{J>0\}$. 

It is worth pointing out that this condition is not always satisfied, but keep
in mind that this fourth step is only used to give a simplest form to the result.

\section{A simple discontinuous example}

In order to introduce discontinuities in the KPP Equation, but also to point out an
interesting feature of the fronts associated to this equation, let us consider a
$1$-d example borrowed from Freidlin's book \cite{F} in which the following model is considered: 
$$ u_t -\frac12 \Delta u = c(x) u(1-u)\quad \hbox{in  }\R\times (0,+\infty)\; ,$$
where $c(x)=c_1$ if $x<1$ and $c(x)=c_2$ if $x\geq 1$. We also assume that $\Ga_0 =(-\infty, 0)$, \ie
the front is located at $x=0$ initially. Concerning function $J$ as in the previous section, 
it is intiialized by $J(x,0)=0$ if $x\leq 0$ while $J(x,0)=+\infty$ if $x>0$ and in the present
case, it satisfies the discontinuous equation
$$ J_t + \frac 12 |D J|^2+ c(x) = 0 \quad\hbox{in }\R^N \times (0,+\infty)\; .$$

For the control formulation for the function $J$, we follow the approach of Part~\ref{part:codim1}:
for $x\in\Omega_1:=\{x<1\}$, we set
$$ \BCL_1 (x,t):=\Big\{\Big(v_1\;,\,0\;,\,-c_1 +\frac{|v_1|^2}2\Big); \ v_1 \in \R^N\Big\}\;.$$
and for $x\in\Omega_2:=\{x>1\}$, we set
$$ \BCL_2 (x,t):=\Big\{\Big(v_2\;,\,0\;,\,-c_2+\frac{|v_2|^2}2 \Big); \ v_2 \in \R^N\Big\}\;.$$
Therefore the cost $-c$ is discontinuous at $x=1$ and 

The following formula allows to compute explicitly function $J$
$$ J(x,t) = \inf\left \{\int_0^t \left(\frac{|\dot y (s)|^2}{2} -c(y(s))\right) ds\  ;
\ y \in H^1(0,t), \ y(0)=x,\ y(t) \leq 0\right\}\; .$$
Notice that, a priori, we should have been careful with this formal formula since the function $c$
is discontinuous at $x=1$ but, at this point of the book, it should be clear for the reader that the
present situation is quite easy to handle, even if we face an unbounded control problem with
an initial data taking value $+\infty$ on $(0,+\infty)$. These two features do not create a
very important difficulty here: the first one because of the simplicity of the Hamiltonian and the
fact that there is no interference between the discontinuity and the gradient term;
the second one is solved by standard arguments, in particular by using the solutions associated with
constant velocities $\max(c_1,c_2)$ and $\min(c_1,c_2)$, \ie solutions of problems without
discontinuity which are sub and supersolutions to our problem. 

Notice that if the trajectory stays on the line $x=1$, the optimal choice consists in choosing
$c(y(s))= \max(c_1,c_2)$, using a tangential dynamic for the trajectory.

\bigskip

From now on, we assume for our purpose that $c_2 > c_1$ and we address the
following question: 
$$\text{\emph{when does the front, starting from $x=0$ reach the value $1$?}}$$
If we just consider the domain $x<1$, the answer should be $t_1=\big(\sqrt{2c_1}\big)^{-1}$ since
the velocity of the front is $\sqrt{2c_1}$ in this domain.

But we may also examine $J(1,t)$ and compute the smallest $t$ for which it reaches zero,
corresponding to the arrival of the front.  It is clear that an optimal trajectory should stay at
$x=1$ on an interval $[0,h]$ and then a straight line to reach $x=0$. Therefore 
$$J(1,t)=\min_{0\leq
    h\leq t}\left(-c_2 h + \frac{1}{2(t-h)}-c_1(t-h)\right)\; .$$
An easy computation gives
$$ J(1,t)= \begin{cases} 
\displaystyle \frac{1}{2t}-c_1t & \hbox{if  } \displaystyle  t\leq \frac{1}{\sqrt{2(c_2-c_1)}}\; ,\\
\sqrt{2}\sqrt{c_2-c_1}-c_2 t & \hbox{otherwise}
\end{cases}\; ,$$
and the front reaches $1$ either at time $t_1=\big(\sqrt{2c_1}\big)^{-1}$ or
$\displaystyle t_2=\sqrt{2}\frac{\sqrt{c_2-c_1}}{c_2}$ if $t_2$ satisfies the constraint
$$ t_2 \geq \frac{1}{\sqrt{2(c_2-c_1)}}\; ,$$
\ie if $c_2\geq 2c_1$.

The reader can check that the inequality $t_2<t_1$ is true for any $c_2> 2c_1$: indeed, if
$X=c_2/c_1$, it is equivalent to $$ X^2-4X+4=(X-2)^2>0\; .$$ In this case, the front looks like the
following picture

\begin{figure}[htp]
   \begin{center}
       \includegraphics[width=0.75\textwidth]{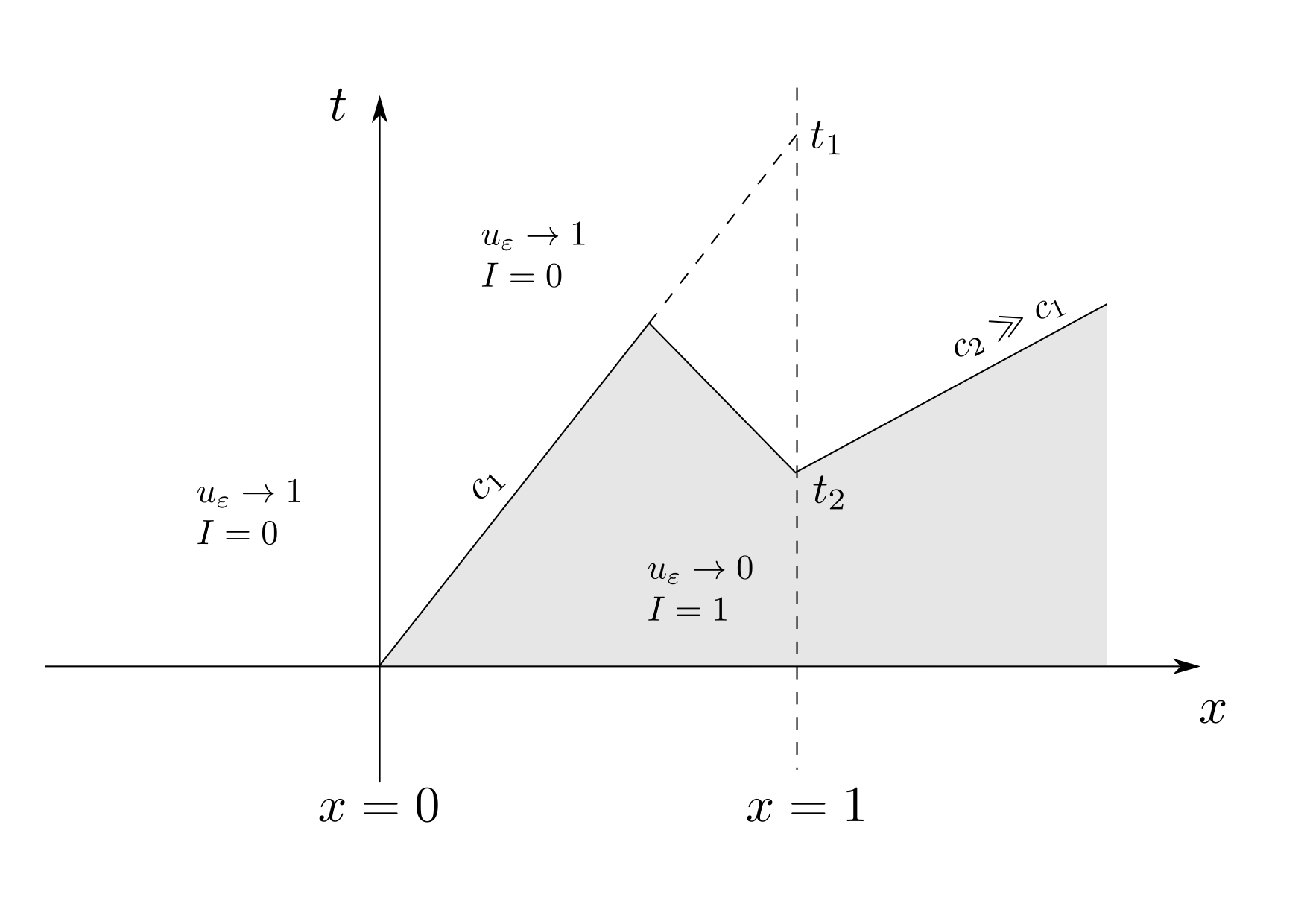}
       \caption{KPP front}
       \label{fig:kpp} 
   \end{center}
\end{figure}

We observe at time $t_2$ a strange phenomenon: a new front is created at 
$x=1$, ahead of the front travelling in $\Omega_1$ with velocity $\sqrt{2c_1}$. This kind of
phenomenon can arise even if $c(x)$ is continuous but the computations are 
easier to describe in the discontinuous setting. We also point out that Freidlin's 
condition holds true in this example.

In the next sections, we first provide results for general KPP Equations in the 
framework of Part~\ref{part:codim1}, \ie  in the case when we have 
discontinuities on an hyperplane. Then we consider some extensions to more 
general type of discontinuities which uses some particular features of the KPP Equations.

\section{The codimension one case}

\index{KPP type problems!half-space case}

With the notations of Part~\ref{part:codim1}, we consider the problem
\begin{equation}\label{KPP-E1}
(\ue)_t -\frac{\e}2 {\rm Tr}(a(x)D^2 \ue)-b(x)\cdot D\ue = \frac{1}{\e}
f(x,\ue)\quad \hbox{in  }\R^N \times (0,+\infty)\; ,
\end{equation}
where, in $\Omega_i$, $a=a^{(i)}$, $b=b^{(i)}$, $f=f^{(i)}$ for $i=1,2$, where $ a^{(i)},
b^{(i)},f^{(i)}$ are bounded Lipschitz continuous functions taking values respectively in $\sym$,
$\R^N$ and $\R$. We assume that the following additional properties hold:

\bigskip

\noindent {\bf (Uniform ellipticity)} \emph{There exists $\nu >0$ such that}
\begin{equation}\label{ueKPP-E1}
a^{(i)}(x)p\cdot p \geq \nu |p|^2 \quad \hbox{for any $x,p \in \R^N$}\; .
\end{equation}

\medskip

\noindent {\bf (KPP-nonlinearity)} \emph{For $i=1,2$ and for any $x \in \Omega_i$: $u\mapsto f^{(i)}(x,u)$
is differentiable at $0$ and for any $u \in [0,1]$}
\begin{equation}\label{knKPP-E1}
\left\{\begin{array}{l}
    f^{(i)}(x,0)= f^{(i)}(x,1)=0, \quad  f^{(i)}(x,u)>0 \quad \hbox{if $0<u<1$}  \\
    \displaystyle c^{(i)}(x)=\frac{\partial f^{(i)}}{\partial u}(x,0)= \sup_{0<u<1} 
    \left(\frac{f^{(i)}(x,u)}{u}\right), 
\end{array}
\right.
\end{equation}
\emph{with $c^{(i)}$ being bounded Lipschitz continuous on $\Omegb_i$. }

\bigskip

Of course, the prototypal example of $f^{(i)}$ is $f^i{(i)}(x,u)
=c^{(i)}(x)u(1-u)$ which is not a globally Lipschitz continuous function of $u$
but since all the solutions $\ue$ will take values in $[0,1]$, this is not a
problem.\\

Next we complement \eqref{KPP-E1} with the initial data
\begin{equation}\label{idKPP-E1}
\ue(x,0)=g(x) \quad \hbox{in  }\R^N\; ,
\end{equation}
where $g : \R^N \to \R$ is a compactly supported continuous function such that
$0 \leq g(x) \leq 1$ in $\R^N$. As above we denote by $\Ga_0$ the support of $g$
which is assumed to be a non-empty compact subset of $\R^N$ satisfying
$$ \overline{\mathrm{Int}(\Ga_0)}=\Ga_0\; .$$
In order to formulate the result, we introduce the following Hamiltonians for
$i=1,2$:
$$ H_i(x,p):= \frac12 a^{(i)}(x)p\cdot p -b^{(i)}(x)\cdot p
+c^{(i)}(x) \; .$$
As we already noticed in the previous subsection, keep in mind that from the
control viewpoint, the cost is $l^{(i)}=-c^{(i)}$.

\begin{theorem}\label{KPP-mr}\ 
    \begin{enumerate}
        \item[$(i)$] As $\e \to 0$, the following convergence holds:
        $$ -\e \log(\ue) \to I \quad \hbox{locally uniformly in  }\R^N \times (0,+\infty)\; ,$$
         where $I$ is the unique solution of
     \begin{equation}\label{vi-KPP}
        \left\{\begin{array}{ll}
            \min(I_t + H_i (x,DI),I)= 0 & \hbox{in  } \Omega_i \times (0,+\infty)\; ,\\[4mm]
        I(x,0) = \left\{\begin{array}{ll}
        0 & \hbox{if $x \in \Ga_0$,}\\ +\infty & \hbox{otherwise}\; ,
        \end{array}
        \right.
        \end{array}
        \right.
    \end{equation}
    associated to the Kirchhoff condition
    \begin{equation}\label{KC-KPP}
        \frac{\partial I}{\partial n_1}+\frac{\partial I}{\partial n_2}=0 \quad \hbox{on  }
        \H \times (0,+\infty)\;.
    \end{equation}
    Equivalently, $I$ is the maximal Ishii solution of 
    variational inequality \eqref{vi-KPP} in $\R^N \times (0,+\infty)$. 
    \item[$(iii)$] As $\e\to0$, the asymptotic behavior of $\ue$ is given by
    $$ \ue(x,t) \to \left\{\begin{array}{ll}
        0 & \hbox{in $\{I>0\}$,}\\
        1 & \hbox{in the interior of the set $\{I=0\}$.}
        \end{array}
        \right.
    $$
    \item[$(iv)$] If Freidlin's condition holds, then $I=\max(J,0)$ where $J$ 
    is either the unique solution of
    \begin{equation}\label{E-KPP}
        \left\{\begin{array}{ll}
            J_t + H_i (x,DJ)= 0 & \hbox{in  } \Omega_i \times (0,+\infty)\; ,\\[4mm]
        J(x,0) = 
        \left\{\begin{array}{ll}
        0 & \hbox{if $x \in \Ga_0$,}\\
        +\infty & \hbox{otherwise}\; ,
        \end{array}
        \right.
        \end{array}
        \right.
    \end{equation}
    associated to the Kirchhoff condition, or equivalently the maximal Ishii
    solution of \eqref{E-KPP} in $\R^N \times (0,+\infty)$. 
    \item[$(v)$] Function $J$ is given by the following representation formula
    $$J(x,t) =\inf \left\{\int_0^t l(y(s), \dot y(s))ds;\  
    y(0)=x,\  y(t) \in G_0,\  y \in H^1(0,t) \right\}\; ,$$
    where $\displaystyle l(y(s), \dot y(s))= \dfrac12 
    [a^{(i)}(y(s))]^{-1}(\dot y(s)-b^{(i)}
    (y(s)))\cdot (\dot y(s)-b^{(i)}(y(s)))-c^{(i)}(y(s))$ 
    if $y(s) \in \Omega_i$ and with 
    the regular control procedure on $\H\times (0,+\infty)$. 
    \end{enumerate}
\end{theorem}

\noindent
We can summarize this result by saying that the ``usual'' KPP-result holds true provided that the
``action functional'' $J$ is suitably defined, taking only regular controls on $\H\times
(0,+\infty)$, using the links between the maximal Ishii viscosity solution, flux-limited solutions
and junction viscosity solutions for the Kirchhoff condition.

\begin{remark}
    As the previous paragraph suggests, the proof of Theorem~\ref{KPP-mr} uses the most
    sophisticated results and tools of Parts~\ref{part:codim1} and \ref{part:NA}, combining the
    different approaches and their connections. We refer the reader to Section~\ref{mg-KPP} where a
    different point of view is described with the aim of treating more general discontinuities. That
    point of view consists in checking whether it is possible to conclude by using only the notion
    of Ishii viscosity solution.  
\end{remark}

\begin{proof} 
    The proof relies on classical arguments which remains valid
    because of the results of Theorem~\ref{IJ-gen-disc} given in Section~\ref{sec:exvi}.
    The aim is make the change of variable
    $$ I_\e (x,t)=-\e \log(\ue(x,t)) $$
    and to show that $I_\e \to I$ locally uniformly in $\R^N \times (0,+\infty)$.  But, in order to
    do so, we first need local uniform bounds on $I_\e$.

    \

    \noindent\textbf{(a)}
    We first notice that, by the Maximum Principle, we have 
    $$ 0 \leq \ue(x,t) \leq 1 \quad \hbox{in }\R^N\times (0,+\infty)\; ,$$
    and therefore $I_\e(x,t) \geq 0$ in  $\R^N\times (0,+\infty)$. In addition, $I_\e$ is
    well-defined because $\ue(x,t)>0$ in  $\R^N\times (0,+\infty)$ by the Strong Maximum Principle.

    Getting an upper bound on $I_\e$ is done by using the trick introduced in \cite{BP2,BP3} (the
    reader can look in those references for the details which follow): we set
    $$ I_\e^A (x,t)=-\e \log\left(\ue(x,t)+ \exp(-A/\e)\right)\; ,$$
    where $A\gg 1$. Then $o(1)\leq I_\e^A (x,t) \leq A$, and it is easy to show that
    $$ \limssup I_\e^A = \min(\limssup I_\e, A)\; . $$
    Therefore controlling $I_\e^A$ uniformly in $A$ provides the same control on $I_\e$.
    Next, using that $f^{(i)} (x,\ue)\geq 0$ in  $\R^N\times (0,+\infty)$, the function $I_\e^A$
    satisfies 
    $$ (I_\e^A)_t -\frac{\e}2 {\rm Tr}(a^{(i)}(x)D^2 I_\e^A)+ \frac12
    a^{(i)}(x)D I_\e^A\cdot D I_\e^A -b^{(i)}(x)\cdot D I_\e^A \leq 0 \quad \hbox{in  }
    \Omega_i \times (0,+\infty)\;,$$
    and the ellipticity assumption together with a Cauchy-Schwartz inequality on the
    $b^{(i)}$-term leads to
    $$ (I_\e^A)_t -\frac{\e}2 {\rm Tr}(a^{(i)}(x)D^2 I_\e^A)+ \frac12 \nu
    |D I_\e^A|^2 \leq k(\nu) \quad \hbox{in  }\Omega_i \times (0,+\infty)\; ,$$
    for some constant $k(\nu)$ large enough, depending only on 
    $\| b^{(i)}\|_\infty$ and $\nu$.

    \

    \noindent\textbf{(b)}
    Passing to the limit through the half-relaxed limits method, setting $\bar I_A=\limssup
    I_\e^A$, we get
    $$ \left\{\begin{array}{ll}
    (\bar I_A)_t + \dfrac12 \nu |D \bar I_A|^2 \leq k(\nu) & \hbox{in  } \R^N \times (0,+\infty)\;
        ,\\[4mm]
    \bar I_A(x,0) = \left\{\begin{array}{ll} 0 & \hbox{if $x \in \Ga_0$,}\\ A & \hbox{otherwise}\; . \end{array}
    \right. \end{array} \right.
    $$
    The Oleinik-Lax formula then implies
    $$ \bar I_A (x,t) \leq \frac{[d(x,\Ga_0)]^2}{2\nu t} + k(\nu)t \quad \hbox{in  }
    \R^N \times (0,+\infty)\; ,$$
    which is the desired uniform bound.

    \

    \noindent\textbf{(c)}
    Therefore we can perform the $I_\e$ change of function and we obtain
    $$ (I_\e)_t -\frac{\e}2 {\rm Tr}(a^{(i)}(x)D^2 I_\e)+ \frac12 
    a^{(i)}(x)D I_\e\cdot D I_\e -b^{(i)}(x)\cdot D I_\e \leq -\frac{f^{(i)} (x,\ue)}{\ue} 
    \quad \hbox{in  }\Omega_i \times (0,+\infty)\; ,$$
    where we have kept the notation $\ue$ in the right-hand side to emphasize the
    role of the quantity $f^{(i)} (x,\ue)/\ue$. Indeed we have both
    $$ - \frac{f^{(i)} (x,\ue)}{\ue} \geq -c^{(i)}(x) \quad \hbox{for any  }x\; ,$$
    and
    $$ - \frac{f^{(i)} (x,\ue)}{\ue} \to -c^{(i)}(x) \quad \hbox{if  }\ue(x,t) \to 0\; ,$$
    and this last case occurs if $I_\e(x,t)$ tends to a strictly positive quantity.

    Using these properties, Theorem~\ref{pro:viscous} implies that $\overline I=
    \limssup I_\e$ and $\underline I=\limiinf I_\e$ are respectively sub and
    supersolutions of the variational inequality (\ref{vi-KPP}) 
    associated with Kirchhoff condition on $\H$.

    \

    \noindent\textbf{(d)}
    In order to conclude, we have just to use Theorem~\ref{IJ-gen-disc}: with the
    notations of this result, we have $$\overline I (x,t) \leq I^+ (x,t) \leq
    \underline I (x,t) \quad \hbox{in  }\R^N\times (0,+\infty)\; ,$$ and, $I^+$
    being continuous, this implies that $I_\e \to I^+$ locally uniformly in
    $\R^N\times (0,+\infty)$. 

    The proof is complete since the other results can be obtained exactly as in the
    standard KPP case.  
\end{proof}

\section{The variational inequality in the codimension~one case}
\label{sec:exvi}

In this section, we study the control/game problems related to the functions
$I$ and $J$ arising in the statement of Theorem~\ref{KPP-mr}, together with the
properties of the associated Bellman Equation or variational inequality.

To do so, we follow the approach of Part~\ref{part:codim1}: for $x\in\Omega_i$,
we set
$$ \BCL_i (x,t):=\Big\{\big(v_i\;,\,0\;,\,l_i(x,v_i)\big); \ v_i \in \R^N\Big\}\;,$$
where 
$$l_i(x,v):=\frac12 [a^{(i)}(x)]^{-1}(v-b^{(i)}(x))\cdot (v-b^{(i)}(x))-c^{(i)}(x)\;.$$
Of course, we are in an unbounded control framework but this does not create any
major additional difficulty, as was already said.

It remains to define the (regular or not) dynamic and cost on $\H=\{x_N=0\}$. So, for $x\in\H$ we
have $(v,0,l) \in \BCL_T(x,t)$ if $v=\alpha v_1+ (1-\alpha) v_2\in\H$ 
and
$$ l=\alpha l_1 (x,v_1) + (1-\alpha)l_2(x,v_2)\; ,$$
The set $\BCL_T^\reg(x,t)$ is defined in the 
same way, adding the condition $v_1\cdot e_N \leq 0$, $v_2\cdot e_N \geq 0$.
Now, if $I_0 \in C_b (\R^N)$, we introduce
$$ J^-(x,t) = \inf_{\Ta}\left\{\int_0^t  l(X(s),\dot X(s)) ds + I_0(X(t)) \right\}\; ,$$
$$ J^+(x,t) = \inf_{\Treg}\left\{\int_0^t  l(X(s),\dot X(s)) ds + I_0(X(t)) \right\}\; ,$$
where, in these formulations, we have replaced $v_i (s) $ ($i=1,2$) or $v(s)$ by
$\dot X(s)$.

In the same way, we introduce
$$ I^-(x,t) = \inf_{\Ta}\sup_\theta \left\{\int_0^{t\wedge \theta}  l(X(s),\dot X(s)) ds + \1_{t < \theta} I_0(X(t)) \right\}\; ,$$
$$ I^+(x,t) = \inf_{\Treg}\sup_\theta\left\{\int_0^{t\wedge \theta}  l(X(s),\dot X(s)) ds + \1_{t < \theta}  I_0(X(t)) \right\}\; .$$

Following the methods of Part~\ref{part:codim1} and~\ref{part:NA}, it is easy to show the following result
\begin{theorem}\label{IJ-gen-cont}\ 
    \begin{enumerate}
        \item[$(i)$]
    The value functions $J^-$ and $J^+$ are continuous and respectively the
minimal Ishii supersolution (and solution) and maximal Ishii subsolution
(and solution) of the equation 
\begin{equation}\label{e-kpp}
J_t + H(x,DJ)= 0 \quad \hbox{in  }\R^N \times (0,+\infty)\; ,
\end{equation}
where $H=H_i$ in $\Omega_i \times (0,+\infty)$ with the initial data
$$ J(x,0)=I_0 (x)  \quad \hbox{in  }\R^N \; .$$

\item[$(ii)$] \SCR holds for the flux-limited problems for Equation~\eqref{e-kpp} with
flux~limiters $\HT$ and $\HTreg$ on $\H$; $J^-$ is the unique flux-limited
solution associated to the flux~limiter $\HT$ and $J^+$ is the unique
flux-limited solution associated to the flux~limiter $\HTreg$. $J^+$ is also the
unique solution associated to the Kirchhoff condition on $\H$.

\item[$(iii)$]
The functions $I^-$ and $I^+$ are continuous and respectively the minimal
Ishii supersolution (and solution) and maximal Ishii subsolution (and solution)
of the equation 
\begin{equation}\label{vi-kpp}
\min(I_t + H(x,DI), I) = 0 \quad \hbox{in  }\R^N \times (0,+\infty)\; ,
\end{equation}
where $H=H_i$ in $\Omega_i \times (0,+\infty)$ with the initial data
$$ I(x,0)=I_0 (x)  \quad \hbox{in  }\R^N \; .$$

\item[$(iv)$] \SCR holds for the flux-limited problems for the variational inequality
\eqref{vi-kpp} with flux~limiters $\HT$ and $\HTreg$; $I^-$ is
the unique flux-limited solution associated to the flux~limiter $\HT$ and $I^+$
is the unique flux-limited solution associated to the flux~limiter $\HTreg$.
$I^+$ is also the unique solution associated to the Kirchhoff condition on
$\H$.
\end{enumerate}
\end{theorem}

In order to treat the KPP problem, we have to extend this result to the case of 
discontinuous $I_0$, with possibly infinite values. Of course, stricto sensu, a
\SCR cannot hold in this case. Indeed, if $u$ and $v$ are respectively a sub and 
supersolution of either \eqref{e-kpp} or \eqref{vi-kpp} with initial data $I_0$,
the inequalities at time $t=0$ are 
$$ u(x,0)\leq I_0^* (x) \quad \hbox{and}\quad v(x,0)\geq (I_0)_* (x)\quad 
\hbox{in  }\R^N\; ,$$
and it is false in general that $u(x,0)\leq v(x,0)$ in $\R^N$. Therefore we have to
extend the meaning of \SCR by saying that a \SCR holds in this context if we have
$$ u(x,t)\leq  v(x,t) \quad 
\hbox{in  }\R^N\times (0,+\infty)\; ,$$
hence for all $t>0$.

With this modified definition, we can formulate a simple result which is exactly
what we need (we do not try to reach the full generality here):
\begin{theorem}\label{IJ-gen-disc} 
    Assume that $\overline{\mathrm{Int}(\Ga_0)}=\Ga_0$, then the results of
    Theorem~\ref{IJ-gen-cont} remain true if $I_0(x) = A \1_{\Ga_0}$ for
    some $A>0$, and even if $A = +\infty$. 
\end{theorem}

\begin{proof} 
    We begin with the case when $A<+\infty$ and we provide the full proof only in the $I$-case, the
    $J$-one being obtained by similar and even simpler arguments.

    \noindent\textsc{Step 1: }
    \emph{Approximation of the data.}    

In order to prove the analogue of $(iii)$, we can approximate $I_0$ by above and
below by sequences $((I_0)^A)_A$ and $((I_0)_A)_A $ of bounded continuous
initial data such that
$$ (I_0)^A \downarrow I_0^* \quad \hbox{and}\quad (I_0)_A \uparrow (I_0)_*\; .$$
We denote by $(I^A)^\pm$ and $(I_A)^\pm$ the minimal and maximal solutions given by 
Theorem~\ref{IJ-gen-cont} with these intial data.

If $u,v$ are respectively a subsolution and a supersolution of the variational
inequality with initial data $I_0$, they are respectively subsolution with
$(I_0)^A$ and supersolution with $(I_0)_A$. Therefore
$$ u\leq (I^A)^+ \quad \hbox{and}\quad (I_A)^- \leq v \quad \hbox{in  }\R^N \times (0,+\infty)\; .$$
It remains to pass to the limit in the variational formulas for $(I^A)^+$ and
$(I_A)^-$. This step is easy for $(I^A)^-$ by the stability of solutions of
differential inclusion (one has just to be careful of the fact that 
we obtain $(I_0)_*$ in the formula at the limit).

For $(I^A)^+$, things are more delicate since we have to deal with regular
trajectories. But here, we can take advantage of the inequality we wish to show
and first argue with a FIXED trajectory (here also one has to be careful because
we obtain $(I_0)^*$ in the formula at the limit).

\noindent \textsc{Step 2: } 
\emph{Both functions $(I^A)^+$ and $(I^A)^-$ are continuous.}

In order to prove the claim, we can use the approach of
the authors in \cite{BC-re}, showing that $I=(I^A)^-$ or $(I^A)^+$ both
satify $$ -\eta (t) \leq I_t(x,t) \leq C\;.$$
for some positive function $\eta$ which may tend to $+\infty$ when $t\to 0$ and 
for some constant $C$. This inequality is obtained by using the arguments of
\cite{BC-re}: we just use a sup-convolution in time
$$ \sup_{0\leq s \leq t}(I(x,s) -\eta (s) (t-s))\; ,$$
and combine it with a comparison result for flux-limited solutions (with the
suitable flux~limiter for $(I^A)^-$ and $(I^A)^+$).

This argument shows that $(I^A)^-$ and $(I^A)^+$ are Lipschitz 
continuous in $x$ (for $t >0)$
where they are strictly positive. 
Indeed, if $I>0$, variational inequality \eqref{vi-kpp} implies that
$H(x,DI)=-I_t\leq \eta(t)$, and the coercivity of $H$ implies a bound on $DI$. 
Then, it is a simple exercice to extend it to all points in 
$\R^N \times (0,+\infty)$, whether $I>0$ or $I=0$.

\noindent \textsc{Step 3: } 
\emph{Strong Comparison Result.}

For the proofs of the \SCR, we still consider $(I_0)^A$, $(I_0)_A$ but the \SCR 
for either $\HT$, $\HTreg$ or the Kirchhoff condition. In the case of $\HTreg$,
for example, we obtain
$$ u\leq (I^A)^+ \quad \hbox{and}\quad (I_A)^+ 
\leq v \quad \hbox{in  }\R^N \times (0,+\infty)\; .$$
To conclude in this case, we have to use Proposition~\ref{prop:extraction.gen}
to pass to the limit by extracting a sequence of trajectories which converges to
a regular trajectory. The case of $\HT$ is simpler.

\noindent \textsc{Step 4: } 
\emph{Passing to the limit to treat the case $A=\infty$.}

In the case where $A=+\infty$, we first notice that all solutions associated
with initial data like $I_0(x) = A \1_{\Ga_0}$, and slightly enlarging or slightly
reducing the set $\Ga_0$ are uniformly locally bounded with respect to $A$
(this can be obtained by choosing appropriate trajectories such as straight
lines). And the limiting function are 
$$ I^-(x,t) = \inf_{\Ta }\sup_\theta\left\{\int_0^{t\wedge \theta}
    l(X(s),\dot X(s)) ds; \ X(t)\in \Ga_0 \right\}\; ,$$
$$ I^+(x,t) = \inf_{\Treg}\sup_\theta\left\{\int_0^{t\wedge \theta}  l(X(s),\dot X(s)) ds; \ X(t)\in
\Ga_0 \right\}\; .$$

Now, if $u$ is a subsolution then, for all $A$ and $C=\max_i (||c_i||_\infty)$,
$\min(u,A-Ct)$ is also a subsolution associated to the initial data $A
\1_{\Ga_0}$. Indeed, since the Hamiltonians are convex, the infimum of two
subsolutions remains a subsolution. We then use the first result to conclude. 
We can use a similar argument for the supersolution, using this time a
comparison with $(I_A)^\pm$, depending of the result we want.  
\end{proof}

\section{Remarks on more general discontinuities}
\label{mg-KPP}

\index{KPP type problems!general domains}

In the proof of Theorem~\ref{KPP-mr}, even if we hide it carefully inside the 
proof of Theorem~\ref{IJ-gen-disc}, we use in an essential way the various 
notions of solutions which are described in Part~\ref{part:NA}, namely \FLS and \JVS 
together with results concerning their links.

This heavy sophisticated machinery is a weakness if we want to address the
case of more general discontinuities for which we are not able to provide such a
precise analysis. Therefore, it is natural to investigate what can be done in those
more general cases.

\subsection{Using the standard notion of Ishii viscosity solution}

In the framework of Chapter~\ref{chap:KPP}, \ie with a codimension $1$ discontinuity
on an hyperplane, the answer is straightforward and this can be seen from two slightly different
points of view:
\begin{enumerate}
        \item[$(a)$] On one hand, in order to conclude, it is enough that the functions $I^+$ and
            $I^-$ appearing in Theorem~\ref{IJ-gen-cont} and Theorem~\ref{IJ-gen-disc} are equal,
            and so the same for  $J^+$ and $J^-$.  Lemma~\ref{lem:H1m.H2p.c} gives conditions under
            which this happens.

    \item[$(b)$] On the other hand, and this is a more general point of view, we can also look for
        conditions under which Ishii viscosity subsolutions are stratified subsolutions (since, as
        always, supersolutions are the same). The conclusion then follows from the comparison result
        for stratified solutions. Since it is easy to see that, on the hyperplane, the
        $\F^N$-inequality on $\H \times (0,\Tf)$ is the $\HT$-one, Lemma~\ref{lem:H1m.H2p.c} still
        gives the answer.
\end{enumerate}

In order to exploit this result, we recall that we have 
$$ H_i(x,p):= \frac12 a^{(i)}(x)p\cdot p-b^{(i)}(x)\cdot p +c^{(i)}(x)\; ,$$
and the computation of $m_1(x,p'), m_2(x,p')$ is easy:
$$ m_i (x,p') = -\frac1{a^{(i)}(x)e_N\cdot e_N} 
\left( a^{(i)}(x)p'\cdot e_N- b^{(i)}(x)\cdot e_N\right)\; .$$
The condition $m_2(x,p') \geq m_1(x,p')$ for any $(x,p')$ which is required in
Lemma~\ref{lem:H1m.H2p.c} in order to have $\HT=\HTreg$ leads to two properties
by using the affine dependence in $p' \in \H$:
\begin{equation}\label{cond-uni-Ishii-KPP-a}
    \frac{a^{(2)}(x)e_N}{a^{(2)}(x)e_N\cdot e_N} - 
    \frac{a^{(1)}(x) e_N}{a^{(1)}(x)e_N\cdot e_N}=0\; ,
\end{equation}
and 
\begin{equation}\label{cond-uni-Ishii-KPP-b}
    \frac{b^{(2)}(x)\cdot e_N}{a^{(2)}(x)e_N\cdot e_N} 
    \geq \frac{b^{(1)}(x)\cdot e_N}{a^{(1)}(x)e_N\cdot e_N}\; .
\end{equation}
Indeed, the inequality $m_2(x,p') \geq m_1(x,p')$ for any $(x,p')$ implies that
the left-hand side of \eqref{cond-uni-Ishii-KPP-a} is colinear to $e_N$ while
its scalar product with $e_N$ is $0$. Notice that in this computation, we have
implicitly assumed that $N\geq 2$ but, if $N=1$ the result remains true with
only \eqref{cond-uni-Ishii-KPP-b}.

Under this condition, Theorem~\ref{KPP-mr} can be proved using only the basic
notion of viscosity solutions.

\begin{remark}
    Recalling that the costs for the associated control problems are
    $$ l_i(x,v)=\frac12 [a^{(i)}(x)]^{-1}(v-b^{(i)}
    (x))\cdot (v-b^{(i)}(x))-c^{(i)}(x)\;,$$ 
    the stronger condition 
    $$ \forall x\in\H\;,\quad 
    b^{(2)}(x)\cdot e_N \geq 0 \geq b^{(1)}(x)\cdot e_N$$
    is quite natural.  Indeed with $b^{(1)}, b^{(2)}$
    pointing towards $\H$, it is clear that a priori regular controls give
    better costs than singular ones.  Condition \eqref{cond-uni-Ishii-KPP-b}
    generalizes this simple case.
\end{remark}

\subsection{Going further with stratified solutions}

Following Section~\ref{sec:Strat-Is} and in particular Proposition~\ref{IequalS}, we can treat more
general situations even if this leads to very restrictive assumptions. 

We consider the following example in the ``cross case'': consider Equations~\eqref{KPP-E1} which
holds in $Q_i \subset \R^2$ where the $Q_i$'s are the four quadrants in $\R^2$, namely
$$ Q_1=\{x_1>0,x_2>0\}\; , Q_2=\{x_1<0,x_2>0\}\; ,\; Q_3=-Q_1\; ,\; Q_4 =
-Q_2\; .$$ 
In order to be able to apply Proposition~\ref{IequalS}, we assume that, for $i=1,2,3,4$, the
$b^{(i)}$ are equal to $0$ and that $a^{(i)}(x)=\lambda^{(i)} (x)Id$ in $Q_i$ for some bounded,
Lipschitz continuous function $\lambda^{(i)}$. We assume also the existence of some constant $\nu
>0$ such that $\lambda^{(i)}(x) \geq \nu$ in $Q_i$ for any $i$.

Then, under natural assumptions on the regularity of the coefficients, the asymptotics of $\ue$ can
easily be obtained in this framework: indeed
\begin{enumerate}
    \item[$(i)$] $\underline I$ is an Ishii viscosity supersolution of the variational inequality in
        $\R^2\times (0,+\infty)$, 
    \item[$(ii)$] $\overline I$ turns out to be a ``stratified subsolution'' of the variational
        inequality in $\R^2\times (0,+\infty)$. Indeed,  on the axes (except $0$), \ie on $\Man{2}$,
        the above analysis shows that $\HT=\HTreg$ inequality holds for $\overline I$ and therefore
        the $\F^2$-one holds too. At $x=0$ for $t>0$, \ie on $\Man{1}$, we clearly have 
        $$ \min( \overline I_t + \max_i(c^{(i)}(x), \overline I)\leq 0\; ,$$
        because all the inequalities $ \min(\overline I_t + c^{(i)}(x),\overline I)\leq 0$ 
        hold by passage to the limit (stability) from the $Q_i$ domain. This is a case where
        Proposition~\ref{IequalS} applies in a very simple way.
\end{enumerate}
 
Hence, $\underline I$ and $\overline I$ are respectively stratified super and subsolutions of the
variational inequality and we can conclude since the comparison result for stratified solutions
easily extend to this framework.

Proposition~\ref{IequalS} allows to treat the following kind of KPP problems: we assume that
$\M=(\Man{k})_{k=0..N}$ is a stratification of $\R^N$ and that in the framework of
Chapter~\ref{chap:KPP}, the $\ue$ are solutions of Equation~\ref{KPP-E1} where
\begin{enumerate}
    \item $a^{(i)}(x)=\lambda^{(i)} (x)Id$ in $\Omega_i$ where the $\Omega_i$ are the connected components 
        of $\Man{N}$. We assume that the functions $\lambda^{(i)}$ are uniformly bounded and Lipschitz 
        continuous functions and there exists a constant $\nu >0$ such that $\lambda^{(i)}(x) \geq \nu$ in
        $\Omega_i$ for any~$i$.
    \item $b=0$ in $\R^N$.
    \item $f=f^{(i)}$ in $\Omega_i$ where the $f^{(i)}$ are KPP-nonlinearities, 
        the $c^{(i)}$ being uniformly bounded and Lipschitz continuous on $\Omegb_i$.
\end{enumerate}

Under these conditions, and if the initial data $g$ is as in Chapter~\ref{chap:KPP}, the result is
the 
\begin{proposition}\
    \begin{enumerate}
        \item[$(i)$] As $\e \to 0$, the following convergence result holds
        $$ -\e \log(\ue) \to I \quad \hbox{locally uniformly in  }\R^N \times (0,+\infty)\; ,$$
        where $I$ is the unique stratified solution of the equation with
        \begin{equation}\label{vi-KPP-II}
            \left\{\begin{array}{ll}
                \min(I_t + H_i (x,DI),I)= 0 & \hbox{in  } \Omega_i \times (0,+\infty)\; ,\\[4mm]
        I(x,0) = 
        \left\{\begin{array}{ll}
        0 & \hbox{if $x \in \Ga_0$,}\\
        +\infty & \hbox{otherwise}\; .
        \end{array}
        \right.
        \end{array}
        \right.
        \end{equation}
    \item[$(ii)$] As $\e\to0$, the following asymptotic behavior holds
        $$ \ue(x,t) \to \left\{\begin{array}{ll}
        0 & \hbox{in $\{I>0\}$,}\\
        1 & \hbox{in the interior of the set $\{I=0\}$.}
        \end{array}
        \right.
        $$
    \item[$(iii)$] If Freidlin's condition holds then $I=\max(J,0)$ where $J$ 
        is the unique stratified solution of
        \begin{equation}\label{E-KPP-II}
        \left\{\begin{array}{ll}
            J_t + H_i (x,DJ)= 0 & \hbox{in  } \Omega_i \times (0,+\infty)\; ,\\[4mm]
        J(x,0) = 
        \left\{\begin{array}{ll}
        0 & \hbox{if $x \in \Ga_0$,}\\
        +\infty & \hbox{otherwise}\; .
        \end{array}
        \right.
        \end{array}
        \right.
        \end{equation}
    \item[$(iv)$] Function $J$ is given by the following representation formula
        $$J(x,t) =\inf \left\{\int_0^t l(y(s), \dot y(s))ds;\  
        y(0)=x,\  y(t) \in \Ga_0,\  y \in H^1(0,t) \right\}\; ,$$
        where $\displaystyle l(y(s), \dot y(s))= \frac12 
        [\lambda^{(i)}(y(s))]^{-1}|\dot y(s)|^2-c^{(i)}(y(s))$ 
        if $y(s) \in \Omega_i$. 
    \end{enumerate}
\end{proposition}

Several remarks on this results
\begin{enumerate}
\item[$(i)$] We have left this result with a slightly imprecise statement, giving the
equations only in $\Man{N}\times (0,\Tf)$ and defining $l$ only in $\Man{N}\times
(0,\Tf)$. The next section will (at least partially) show why this is enough.
\item[$(ii)$] As above in the ``cross case'', the proof that $\overline I$ is a
stratified subsolution comes from the arguments given in the next section.
\item[$(iii)$] The first part of this result holds for example in the counter-example
in dimension $1$ given in Chapter~\ref{chap:KPP}, the only point is that
Freidlin's condition is not satisfied.
\end{enumerate}

%
%

\chapter{Dealing with jumps}
\label{chap:jumps} 
\index{Jumps}

\abstract{Almost everywhere in this book, a key assumption is that the problem at hand can be
``localized'', in particular in the comparison proof. The question of jumps which, in
addition, does not seem so consistent with the ``regularity of subsolutions'', may appear as being
completely out of reach. These questions are discussed here: some problems with jumps are already
treated by the results of this book; some are more tractable than one may think, including some
strange quasi-variational inequalities like those which arise in a recent article by Bouin, Calvez,
Grenier and Nadin \cite{BCGN}.}

The aim of this chapter is to know whether the formalism of Chapter~\ref{chap:control.tools} allows
to deal with control problems involving jumps of the trajectories, and what kind of problems can be
solved at the pde level.

In  Chapter~\ref{chap:control.tools}, the dynamic-discount-cost is defined by
$$(\dot X,\dot T,\dot D, \dot L)(s)=(b,c,l)(s)\in\BCL(X(s),T(s))\;,$$
but notice that in this differential inclusion, variable $s$ is actually an ``artificial time'', used
to describe the state of the system $x=X(s)$ at the ``real time'' $t=T(s)$. 

Our assumptions on the set $\BCL$ allow the possibility that $b^t(s)=0$ on some interval $[s_1,s_2]$,
leading to $\dot T(s)=b^t(s)=0$ there. In that situation, the behavior of $X$ can be interpreted as
a jump with respect to the ``real time'' variable: while the trajectory $s\mapsto (X(s),T(s))$ remains
continuous, at time $t=T(s_1)=T(s_2)$, we observe a jump for the spatial trajectory $X$ from
$X(s_1)$ to $X(s_2)$. Figure~\ref{fig:jump} below illustrates the situation in the $(X,T)$ plane.

It is then clear that formally at least, our framework allows trajectories with jumps.

In the next sections, we present three interesting examples involving jumps which can be fitted into
our framework, under some assumptions:
\begin{enumerate}
    \item[$(i)$] an obstacle problem which readily fits into the framework;
    \item[$(ii)$] a quasi-variational inequality, which seems a priori way out of our reach
        but turns out to be tractable under some assumption;
    \item[$(iii)$] a large deviation problem with jumps, implying a more complexe quasi-variational
        inequality that we manage to solve thanks to a series of ``miracles''.
\end{enumerate}

\begin{figure}[htp]
   \begin{center}
   \includegraphics[width=0.6\textwidth]{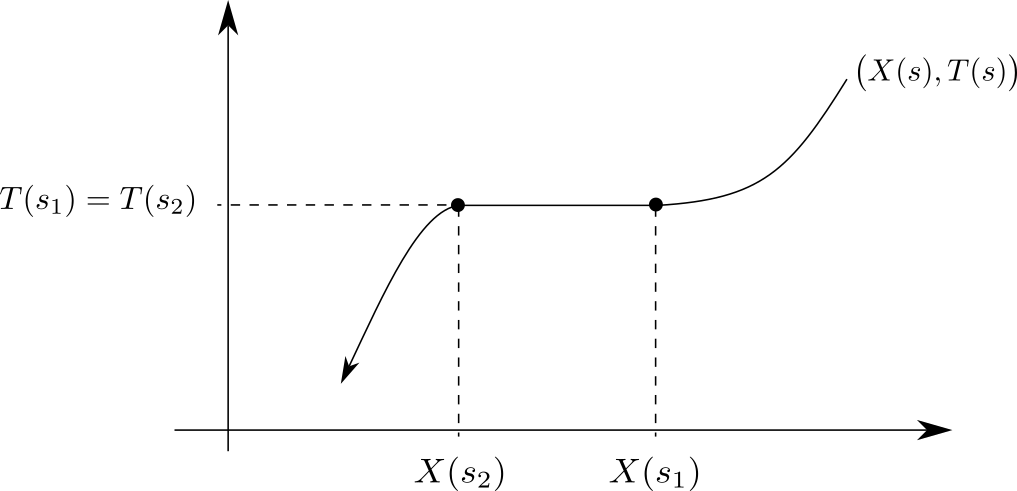}
   \caption{Jump in a controled trajectory}
   \label{fig:jump} 
   \end{center}
\end{figure}

\section{A simple obstacle problem}

Let us examine a ``pure jump'' situation where $b^x=v \in \overline{B(0,R)}\subset \R^N$,
$b^t\equiv 0$, $c\equiv 0$ and $l=l(v)$. More precisely, $\BCL(x,t)$ is independent of $x$ and $t$
and reduces to
$$ \BCL=\Big\{ \big((v,0),0,l(v)\big):v\in\overline{B(0,R)}\Big\}\;.$$
We assume that $l(0)=0$ and the first consequence of the convexity of $\BCL$ is that $l(\alpha
v)=\alpha l(v)$ for any $v\in \overline{B(0,R)}$. Another assumption which avoids oscillating
trajectories is that if $v=v_1+v_2$ then $l(v)\leq l(v_1)+l(v_2)$. As a consequence of these two
properties, $l$ is a convex function of $v$. 

Now we turn to the resolution of a very simple obstacle problem, which looks like a $\F_{init}$
problem: if $u_0$ is a continuous function, let us solve
$$ \max(h(D_xu), u-u_0)=0 \quad\hbox{in  }\R^N\; ,$$
where $h(p_x)= \max\limits_{|v| \leq R}\left(-v\cdot p_x-l(v)\right)$. 
The control interpretation suggests the solution
$$ u(x):= \inf\left\{ \int_0^\theta l(\dot X(s)) ds + u_0(X(\theta)): 
\ X(0)=x,\ |\dot X(s)|\leq R, \ \theta >0 \right\}\;,$$
but by Jensen's inequality and the homogeneity of $l$
$$  \int_0^\theta l(\dot X(s)) ds \geq \theta l \left (\frac 1 \theta \int_0^\theta \dot X(s) ds\right)= 
\theta l \left (\frac 1 \theta  (X(\theta)-x)\right) = l \left (X(\theta)-x\right) \;.$$
Therefore,
$$ u(x)= \inf \left\{ u_0(X(\theta))+l \left (X(\theta)-x\right): \ X(0)=x,
\ |\dot X(s)|\leq R, \ \theta >0 \right\}\; ,$$
or equivalently
$$ u(x)= \inf_{y\in \R^N} \Big\{ u_0(y)+l(y-x) \Big\}\; ,$$
which can be interpreted as the minimal value which can be obtained by making a jump from $x$ to
$y=X(\theta)$ with a cost $l(y-x)=l \big (X(\theta)-x\big)$ for this jump.

This very simple example gives an idea of the type of jumps which can be taken into account by the
framework of Chapter~\ref{chap:control.tools}. The next section examines cases which may not enter
into this framework but which can be handled.

\section{Quasi-variational inequalities}

In the control literature, jumps arise in particular in inventory management and lead to
quasi-variational inequality (QVI in short). We refer the reader to Bensoussan and Lions
\cite{BL-QVI} for a study of such QVI in the framework of stochastic control/elliptic-parabolic
pdes, which was the first situation where they were studied. 

In their book, the jumps play a role via an operator $\mathcal{M}$ which is typically of the form 
$$
\mathcal{M}u(x):=\min_{\xi \in \Xi}\left(u(x+\xi) + k +C(\xi)\right)\quad \hbox{for  }u\in
C_b(\R^N)\;,
$$
where $\Xi$ is a bounded or unbounded subset of $\R^N$, $k\geq 0$ is fixed cost and $C$ is a cost
depending on the size of the jump. In general, one assumes that $C$ is a continuous function such
that $C(\xi)\geq 0$ and $C(0)=0$ if $0\in \Xi$.  If $\Xi$ is unbounded, it is generally assumed that
$C$ is coercive. The typical case which is studied in \cite{BL-QVI} is when $\Xi=[0,+\infty)^N$ and
$C$ satisfies the following sublinearity assumption
\begin{equation}\label{eqn:sublin}
C(\xi_1+\xi_2)\leq C(\xi_1)+C(\xi_2)\quad\hbox{for any  }\xi_1,\xi_2\in  \Xi\; .
\end{equation}
This hypothesis means that ``one large jump is better than two smaller ones'' and its main technical
interest is to avoid the accumulation of very small jumps.

In the QVI, the complete Hamiltonian $\F$ takes the form $\max(\,\cdots, u-\mathcal{M}u)$ and as a
by-product of this form or of the control problem, one gets immediately the inequality $u\leq
\mathcal{M}u$ in $\R^N$ or $\R^N\times[0,\Tf]$.

\subsection*{The case $k>0$}

A favorable situation is when $k>0$, both for second-order HJB Equations as in \cite{BL-QVI} but
also for first-order HJ-Equations: we refer for example to \cite{GB-qvi-cont,GB-qvi-Eq} for simple
ideas to treat such QVI in the continuous framework, both from the control and pde points-of-view.
We point out that the classical comparison results of the continuous case extend without any change
of assumptions to the QVI-case.

In the discontinuous framework, and in particular in the stratified one where the localization of
the comparison proof seems unavoidable, the situation seems hopeless: how a local proof could give
the result for a nonlocal equation? We stress that here, not only do the local values of the sub and
supersolution play a role, but also those associated to points where the process can jump to. So, from a
purely technical perspective, the localization methods of we developed in Section~\ref{sect:htc}
seem a priori unadapted. 

We want however to make a simple remark concerning a particular case: assume that $C$ is a Lipschitz
continuous, coercive function satisfying \eqref{eqn:sublin}. An immediate consequence of
\eqref{eqn:sublin} is that
\begin{equation}\label{ineq:sublin2}
    -C(x)\leq -C(x+\xi)+C(\xi)\;.
\end{equation}
Using a standard regularization of $C$ with a sequence of positive, compactly supported
mollifying kernels $(\rho_\e)_\e$, we get a $C^1$-function $-C_\e:=-C\star \rho_\e$ satisfying also
\eqref{ineq:sublin2}, \ie $ -C_\e (x)\leq \mathcal{M}(-C_\e) (x)-k$.

As a consequence, for $K>0$ large enough, the function $-C_\e(x)-Kt$ provides a strict subsolution
of the QVI, which is as natural as it could be since we use the cost function $C$ in an essential
way.

Of course, in general $C$ may be defined only on $\Xi$ and its extension to $\R^N$ can be a problem.
For example, if $\Xi=[0,+\infty)^N$, it is not completely clear how to do it without additional
assumption but if $C(\xi)=C(\xi_1,\cdots,\xi_N)$ is increasing \wrt $\xi_i$ for all $i$, setting
$$C(\xi_1,\cdots,\xi_N)=C(|\xi_1|,\cdots,|\xi_N|)\; ,$$
for all $(\xi_1,\cdots,\xi_N)\in \R^N$ solves the problem.

If the difficulty associated to the localization is solved---either as above or by an other
argument---and we succeed to build a {\em strict subsolution} $u$, the following argument which
plays a key role in the case of continuous equations should also give the answer (we drop the time
variable for simplicity here):

\

\begin{minipage}{0.9\textwidth}
    \emph{If $v$ is the supersolution to be compared to $u$ and if $\xb$ is a maximum point
    of $u-v$, we cannot have $v(\xb) \geq \mathcal{M}v(\xb)$.}
\end{minipage}

\

Indeed, since $u$ is a strict subsolution we have $u(\xb) \leq \mathcal{M}u(\xb)-\delta$ for some
$\delta >0$ and therefore, by the form of $\mathcal{M}$, 
$$ u(\xb)- v(\xb)\leq \mathcal{M}u(\xb)-\mathcal{M}v(\xb)-\delta\leq \max(u-v)-\delta\; ,$$ 
a clear contradiction with the definition of $\xb$. Moreover, this argument is ``robust'' in the
sense that, if $u$ is regularized into $u_\e$ and $x_\e$ is a maximum point of $u_\e-v$, the above
argument also applies at $x_\e$ for $\e$ small enough. Hence, either $v(\xb) < \mathcal{M}v(\xb)$ or
$v(x_\e) < \mathcal{M}v(x_\e)$ and it remains to show that the DPP holds for $v$ at $\xb$ or $x_\e$
{\em without any jump} for a small time interval $[0,\tau]$, a not so difficult task. Using all
these ingredients, we recover all the inequalities for $u$ (or $u_\e$) and $v$ which allow to can
argue as usual. 

We admit that the arguments above are presented a little bit formally, but we believe that they
provide the answer in ``reasonable cases''.

\subsection*{The case $k=0$}

This situation is more complicated and even hopeless if $\min_{\xi \in \Xi}\left(
C(\xi)\right)=0$, in particular if $0\in \Xi$: indeed, in this case, no comparison can hold since,
for any constant $\bar c$, we have 
$$\bar c - \mathcal{M}\bar c = \bar c-\min_{\xi \in \Xi}\left( \bar c +C(\xi)\right)
=-\min_{\xi \in \Xi}\left( C(\xi)\right)=0\;,$$
so that all constants are supersolutions.

Therefore, either we are more or less back to the case when $k>0$ if $0\notin \Xi$; or we have to
reinterpret the QVI in terms of gradient constraints---as we did in the previous section---in order
that these cases fit into the theory. This means that we should have the classical sublinearity
assumption together with the homogeneity of degree~$1$.

Obviously we are not going to study these cases in details but we want to point out that jumps can
help by ensuring the regularity of subsolutions on $\Man{k}$ if $0\in \Xi$: indeed subsolutions of
QVI satisfy
$$ u(x) \leq u(x+\xi) + C(\xi) \quad \hbox{for all }\xi \in \Xi\; ,$$
and provided there exists a sequence $(\xi_\e)_\e$ converging to $0$ such that $x+\xi_\e \notin \Man{k}$,
we get the regularity on $\Man{k}$. For $\Man{N}$, it is enough to adapt this assumption to have the
regularity from both sides, which is in particular true if
$B(0,\eta)\subset \Xi$ for some $\eta >0$.

In the next section, we present an example which is almost entering in the stratified framework and
for which an additional information allows to prove the comparison result.

\section{A large deviations problem involving jumps}

The aim of this section is to examine an HJ-problem which appears in Bouin, Calvez, Grenier and
Nadin \cite{BCGN}; its formulation is highly non-standard and seems rather far from what we are
doing in this book but we show how to analyze the different difficulties in light of the stratified
approach.

The problem consists in looking for a function $u: [0,\Tf)\times \R^N \times\R^N\to \R$,
solution in $(0,\Tf)\times\R^N\times\R$ of 
$$\begin{cases}
    \max\Big(u_t  (t,x,v) +v\cdot D_x u(t,x,v)-1\;,\, u(t,x,v)- m(t,x) -|v|^2\Big)=0\;,\\[2mm]
    m_t (t,x)\leq 0 \quad \hbox{and} \quad m_t (t,x)= 0\quad \hbox{if  } \mathcal{S}(t,x)=\{0\}\;,
\end{cases}$$
where $m(t,x)=\min_{v'} u(t,x,v')$ and $\mathcal{S}(t,x)$ is the set of all $v'$ where this min is
achieved. These equations are complemented by an initial data 
$$u(0,x,v)=u_0(x,v)\quad \hbox{in  } \R^N\; ,$$
where $u_0$ is a continuous function such that $u_0(x,v)-|v|^2$ is bounded.

\subsection*{Analysis of the problem}

In order to analyze this problem, it is more convenient to consider $w(t,x,v)=u(t,x,v)-|v|^2$ which
is expected to be a bounded continuous function and which first solves
$$ \max\Big(w_t  (t,x,v) +v\cdot D_x w(t,x,v)-1\;,\, w(t,x,v)- \mathcal{M}w(t,x)\Big)=0$$ 
in $(0,\Tf)\times \R^N\times \R^N$ where 
$$ \mathcal{M}w(t,x):=\min_{v'}\big(w(t,x,v')+|v'|^2\big)\; .$$
This equation for $w$ generates several remarks: of course, it looks like a quasi-variational
inequality presented in the previous section, \ie 
$$ \mathcal{M}w(t,x):=\min_{v'}\big(w(t,x,v')+C(v')\big)\; ,$$
with $C(v')=|v'|^2$. But here the function $C$ has all sorts of disadvantages: it is not sublinear,
nor homogeneous of degre $1$ and $ \min_{v'} C(v')=0$. This causes a problem for the initial data,
\ie for the $\F_{init}$-equation,
\begin{equation}\label{eq:id-qvi-c}
\max\Big(w(0,x,v)-w_0(x,v)\;,\, w(0,x,v)-\mathcal{M}w(0,x)\Big)=0\quad\hbox{in  }\R^N \times \R^N\; ,
\end{equation}
where $w_0(x,v)=u_0(x,v)-|v|^2$. Indeed, any constant is a supersolution and therefore the
$\F_{init}$-equation does not determine uniquely $w(0,x,v)$ and it seems we are in the worst
scenario possible because there is no way that the above equation could fit into the control
framework we have described in Chapter~\ref{chap:control.tools}.

We are going anyway to push the arguments as far as possible, in order to show that we can also take
advantage of some features of the QVI along the lines of the remarks we did at the end of the
previous section. We are also going to forget the problem with the initial data, by assuming that it
is achieved in the classical way: solving formally \eqref{eq:id-qvi-c}, which consists here in
taking the maximal subsolution, the ``natural'' initial data should be 
$w(0,x,v)=\min\big(w_0(x,v),\mathcal{M}w_0(x)\big)$. 

\bigskip

For starters, we remark that the equation implies
\begin{equation}\label{eqn:qvi-reg}
w(t,x,0)\leq  \mathcal{M}w(t,x)\leq w(t,x,v')+|v'|^2\quad\hbox{for any  } v'\; ,
\end{equation}
and therefore the min is always achieved for $v'=0$. Since $w(t,x,v')+|v'|^2=u(t,x,v')$, we have
$w(t,x,0)=m(t,x)$ and 
$$\mathcal{S}(t,x)=\{v' ;\; w(t,x,v')+|v'|^2=\mathcal{M}w(t,x)\}\; .$$
We deduce two properties from this remark: on one hand, on $\Man{N}:=\{(t,x,v);\ v=0\}$, the
the stratified inequality holds:
$$ w_t(t,x,0) \leq 0\quad \hbox{in  }(0,\Tf)\times \R^N\; ,$$
and, on the other hand, we get the unusual supersolution inequality
$$ w_t(t,x,0) \geq 0\quad \hbox{if  } \mathcal{S}(t,x)=\{0\}\; ,$$
where $\mathcal{S}(t,x)$ is defined above in terms of $w$ and $\mathcal{M}$.

At this level of the analysis, we face a problem which cannot be formulated as a standard control
problem satisfying the \HBCL assumptions. But apparently, all the correct stratified inequalities on
$\Man{N}$ are available. Moreover, $\Man{N}$ seems to be a discontinuity for the cost since the term
$w_t (t,x,v) +v\cdot D_x w(t,x,v)-1$ in the equation is associated to a cost $1$ while the $
w_t(t,x,0) \leq 0$ suggests a cost $0$ on $\Man{N}$.

Now we turn to the standard assumptions in the stratified framework, namely \TC and \NCe. Concerning
\TC, the Hamiltonian $p_t+v\cdot p_x -1$ satisfies \TCs and it can easily be seen that the term
$w(t,x,v)- \mathcal{M}w(t,x)$ does not cause any problem for tangential regularization; we can
even remark that a regularization in $(t,x)$ can be performed even far from $\Man{N}$, allowing to
assume that the subsolution are smooth in $t$ and $x$. For \NCe, \eqref{eqn:qvi-reg} gives more than
needed.

Hence, we can almost perform the proof of Theorem~\ref{comp-strat-RN} except two additional
difficulties: on one hand, since we are not in a standard control framework, we cannot use
Lemma~\ref{lem:comp.fundamental}. On the other hand, the localization arguments are more tricky
to apply here.

\

\noindent\emph{Complete failure? Not yet!}

\subsection*{Sketching the comparison result}

If $w_1$ and $w_2$ are respectively sub and supersolutions of the above problem with $w_1(0,x,v)\leq
w_2(0,x,v)$ in $\R^N\times \R^N$, we consider, for some parameters $0<\mu<1$ close to $1$, $\delta,
\alpha, \eta >0$ small enough,
$$M= M(\mu,\delta, \eta):=\max_{[0,\Tf]\times \R^N\times \R^N}\Big(\mu w_1(t,x,v)-w_2(t,x,v)- 
\delta |v|^2 -\alpha (|x|^2+1)^{1/2}-\eta t \Big)\; .$$
If $M>0$, then the max cannot be achieved for $t=0$ \footnote{By adding some large positive constant to
$w_1$ and $w_2$, we may assume \wlg that $w_1,w_2\geq 0$ in $[0,\Tf]\times \R^N\times \R^N$.}. Now,
if $(t,x,v)$ is a maximum point, there are two cases.

\

\noindent\textsc{Case 1.---} If $w_2(t,x,v) < \mathcal{M}w_2 (t,x)$, the conclusion follows easily
since $(i)$ because of $t>0$, we can assume \wlg that $w_1$ is smooth in $t$ and $x$; $(ii)$ for
$\alpha$ small enough compared to $\eta$, $\mu w_1(t,x,v) -\alpha (|x|^2+1)^{1/2}-\eta t$ is a
strict, smooth subsolution of $w_t  (t,x,v) +v\cdot D_x w(t,x,v)-1=0$; $(iii)$ this smooth
subsolution is a test-function for $w_2$.

        \

\noindent\textsc{Case 2.---} If $w_2(t,x,v) \geq \mathcal{M}w_2 (t,x)= w_2(t,x,v')+|v'|^2$, using
that $w_1(t,x,v)\leq w_1(t,x,v')+|v'|^2$, we have
$$ M \leq \mu(w_1(t,x,v')+|v'|^2)-(w_2(t,x,v')+|v'|^2)- \delta |v|^2 -\alpha(|x|^2+1)^{1/2}-\eta t\; ,$$
and, if $\delta < (1-\mu)$
\begin{align}
 M & \leq \mu w_1(t,x,v')- w_2(t,x,v')- (1-\mu) |v'|^2 - \delta |v|^2 -\alpha (|x|^2+1)^{1/2}-\eta t \\
   &\leq   \mu w_1(t,x,v')- w_2(t,x,v')- \delta |v'|^2 - \delta |v|^2 -\alpha (|x|^2+1)^{1/2}-\eta t\\
   &\leq  M-\delta |v|^2\; .
\end{align}
Hence, necessarily $v=0$ but examining more carefully the above inequalities and using $\delta <
(1-\mu)$, we can also deduce that $v'=0$. Hence $\mathcal{S}_2(t,x)=\{v' ;\;
w_2(t,x,v')+|v'|^2\}=\mathcal{M}w_2(t,x)\}=\{0\}$ and $(w_2)_t (t,x,v) \geq 0$. There, we reach a
contradiction since $\mu w_1(t,x,v) -\alpha (|x|^2+1)^{1/2}-\eta t$ is a strict, smooth subsolution
of $w_t (t,x,0)=0$.

And the sketch of the proof is complete.
 
\bigskip

The reader may think---and he/she would be right---that the above proof works because of a
succession of miracles: it is clear that the ``$\mu$-trick'', rather classical in this
QVI-framework, allows to overcome in a perfect way the difficulty due to the non-standard features
of the QVI by leading us to the exact situation where we can use the exotic supersolution property
on $v=0$.
  
But the above example is a rare case where some $b^t=0$-controls play a key role while the above
analysis shows, as was already mentioned, that these jumps can easily be taken into account---in
particular in \NCe.

\chapter{On Stratified Networks}
\label{chap:networks}
\abstract{The case of Hamilton-Jacobi Equations on general networks is left aside of this book in
order to keep its length somehow ``reasonable''(!). Here, a sketch of what could be done in the
case of ``stratified network'' is given.}

We recall that a simple network in $\R^2$ is a set containing points, also called nodes, connected
by segments (called edges). Typical examples are a map with roads or highways connecting cities,
cross-roads etc. In this most simple framework, edges are one-dimensional
objects but, of course, more complicated situations can be considered.

One can define control problems and Hamilton-Jacobi Equations on such networks and
Part~\ref{part:NA} is strongly inspired by the theoretical works of Imbert and Monneau
\cite{IM,IM-md,IN}, and Lions and Souganidis \cite{LiSo1,LiSo2} for treating various junctions
conditions at nodes. This may give an idea of what can be done in this direction. Several works have
also been devoted to consider applications, and in particular to traffic problems. We refer the
reader to Imbert, Monneau and Zidani \cite{IMZ}, Forcadel and Salazar \cite{MR3385134}, Forcadel,
Salazar, Wilfredo and Zaydan \cite{FSZ}, but our list is far from being complete and up-to-date. We also recall
that multi-dimensional networks were considered in Achdou, Oudet and Tchou \cite{AOT,AOTbis}, Imbert and Monneau
\cite{IM-md} for all dimensions.
\index{Applications!traffic problems}

In this chapter investigate the fact that the framework of stratified problems in Whitney
stratifications can lead to a rather general point of view of networks, connecting manifolds of various
dimensions. However, we are going to restrict ourselves to a simplified situation.

\section{Stratified networks by penalization}
\index{Stratified networks}

Let us consider a time-independent stratification $\M=(\Man{k})_k$, \ie of the form
$$ \Man{k}= \tMan{k-1}\times \R \quad \hbox{for  }k\geq 1\; ,$$
where $(\tMan{k})_k$ is a \AFS of $\R^N$. 
The \emph{stratified network} we consider here is the following:
$\mathbf{N}:=\bigcup_{k=1}^{N}\Man{k}$. In other words, this network contains manifolds of strictly
positive codimensions only, no open set of $\R^N\times\R$.

Now, we introduce the value function $U^\e:\R^N \times [0,\Tf]\to \R$ defined in the framework of
Chapter~\ref{chap:stratcontr}, replacing $(b,c,l)\in \BCL(x,t)$ by 
$$\big(b,c,l+\e^{-1}\dN(x)\big)\in\BCL_\e(x,t)\;,$$ 
where function $\dN(\cdot)$ denotes the distance to $\tMan{0}\cup\cdots
\cup\tMan{N}$. Obviously this change of cost has the objective to make more and more expensive an
excursion in $\Man{N+1}$ and therefore to force the trajectories to remain on $\mathbf{N}$.

We recall that, by Chapter~\ref{chap:stratcontr} and if suitable assumptions are satisfied, the
value function $U^\e$ is continuous and the unique solution of $\F=\e^{-1}\dN(x)$ in $\R^N \times
[0,\Tf]$ with $\Fk(x,t,U^\e,DU^\e)=0$ on $\Man{k}$ for $k<N+1$.

\bigskip

Our main aim is to provide the asymptotic behavior of $U^\e$ and to do so, we face several problems: the
first one has to do with the half-relaxed limits method. Of course, we have $\limssup U^\e (x,t)=+\infty$ on
$\Man{N+1}$ and, with the standard definition, this would imply that $\limssup U^\e (x,t)=+\infty$
for any $(x,t)\in \mathbf{N}$. Hence, we have to modify the definition of the $\limssup$ in order to
take into account only the points~in~$\mathbf{N}$.

The next difficulty is with the initial data and, to avoid a tedious discussion, we are going to
assume that we know that 
\begin{equation}\label{nopbid}
\overline U (x,0):=\limssup U^\e (x,0) \leq u_0(x) \leq \underline 
    U (x,0):=\limiinf U^\e (x,0) \quad \hbox{for all  }x\in \R^N\; ,
\end{equation}
for some function $u_0 \in C(\R^N)$. This assumption is automatically satisfied if, for example,
$\F_{init}$ reduces to $r-u_0(x)$. We refer to Part~\ref{S-BC} where we already dealt with this
specific topic. 

The last one is related to the regularity of $\overline U:=\limssup U^\e$: indeed the subsolution
$\overline U$ is just defined on $\mathbf{N}$ and we no longer have a $\F_*$-inequality to ensure
the regularity. A priori it could be possible that, if $(x,t) \in \Man{k}$ is on the boundary of
some connected component $\Man{k'}_i$ of $\Man{k'}$ for some $k'>k$, $\overline U$ would not be
$\Man{k'}_i$-regular. This difficulty that we possibly face on the boundary of a connected
component $\Man{k'}_i$ of the $\Man{k'}$ for $k<N+1$ is rather closed to the one we encounter on
$\domeg \times (0,\Tf)$ for state-constrained problems.

To overcome this difficulty, we have to connect the different parts of the network to avoid a
completely different behavior of $U^\e$ on them and to do so, let us first introduce the right space
of test-functions:
\begin{definition}
    A function $\psi$ is in ${\rm PC}^1(\mathbf{N})$ if $\psi$ is continuous on $\mathbf{N}$ and
    $\psi$ is $C^1$ on each $\Man{k}$ for $1\leq k \leq N$.  
\end{definition}
Notice that clearly, such test-functions are well-adapted to stratified subsolutions on networks
since we have to check inequalities on $\Man{k}$ for $1\leq k \leq N$.

Next we give the
\begin{lemma}\label{lem:onlyMk}
    Under assumptions \HBCL, \TCBCL, \NCBCL, then for any $1\leq k \leq N-1$ we have
    $$ \limssup U^\e = \limssup \left(U^\e|_{\Man{k}}\right)\quad \hbox{on  }\Man{k}\; .$$
\end{lemma}

The interest of this lemma is clear: the values of $\overline U=\limssup U^\e$ on $\Man{k}$ are
obtained by using only points on $\Man{k}$; this prevents the values of $\overline U$ on $\Man{k}$
to depend on the nearby different connected components $\Man{k'}_i$.

\begin{proof}
    Let $(x,t)$ be a point in $\Man{k}$. We first assume that there exists $\psi\in {\rm
    PC}^1(\mathbf{N})$ such that $(x,t)$ is a local strict maximum point on $\mathbf{N}$ of
    $\overline U - \psi$. The point $(x,t)$ is also a local strict maximum point on $\mathbf{N}$ of
    $\overline U  - \psi-Cd(x,\Man{k})$ for any $C>0$ and, by the standard properties of the
    $\limssup$, there is a sequence $((x_\e, t_\e))_\e$ of local maximum points on $\mathbf{N}$ of
    $U^\e - \psi-Cd(x,\Man{k})$ such that $(x_\e, t_\e)\to (x,t)$ and $U^\e(x_\e, t_\e)\to \overline
    U (x,t)$.  But, if $C$ is large enough, the normal controllability assumption implies that
    $(x_\e, t_\e)$ cannot be on $\Man{k'}$ for $k'>k$ since the $\F^{k'}$-inequality cannot hold (we
    recall that the distance function to $\Man{k}$ is smooth outside $\Man{k}$). Hence
    $(x_\e,t_\e)\in \Man{k}$ and the claim is proved for such points.

    It remains to prove the claim for points for which, a priori, such a function $\psi$ does not
    exist. This can be done classically by looking at the function $$(y,s) \mapsto \overline U
    (y,s)- \frac{|y-x|^2}{\alpha}-\frac{|s-t|^2}{\alpha}-C_\alpha d(x,\Man{k})$$ which necessarily
    achieves its maximum on $\Man{k}$ if $C_\alpha$ is large enough by the same argument as above.
    At any maximum point $(y_\alpha,t_\alpha)$, we have the desired property and since
    $(y_\alpha,t_\alpha)\to (x,t)$ with $\overline U (y_\alpha,t_\alpha)\to \overline U (x,t)$, the
    result easily follows.
\end{proof}

Anyway, Lemma~\ref{lem:onlyMk} is not sufficient to get the regularity for which we add the assumption

\begin{assumption}{\IDPN}{Inward-pointing Dynamic Property for a Network.}\label{page:IDPN} 
    For any $1\leq k \leq N$ and for each connected component $\Man{k}_i$ of $\Man{k}$, the
    assumptions of Lemma~\ref{lem:sufficient.cone} hold true for any $(y,t) \in
    \overline{\Man{k}_i}$ with $\Omega$ replaced by $\tMan{k}_i$ seen as a domain in $\R^{k-1}$.
\end{assumption}

Thanks to Lemma~\ref{lem:sufficient.cone} and Remark~\ref{rem:ineg-sub-cone-int}, an immediate
consequence of \IDPN is the 
\begin{lemma}
    For $1\leq k \leq N-1$, if $(y,t) \in \partial \Man{k}_i$ for some connected component of
    $\Man{k}$, then there exist $r, \overline M >0$ and a continuous function $b$---all depending on
    $(y,t)$---such that 
    \begin{equation}\label{transm-netw}
        - b(x,t)\cdot D U^\e\leq {\overline M}\quad \hbox{on}\quad \overline{\Man{k}_i}\cap B((y,t),r)\; .
    \end{equation}
\end{lemma}

Now we can state our result
\begin{theorem}\label{thm:network}
    Under assumptions \HBCL, \TCBCL, \NCBCL and \IDPN, the value functions $U^\e$ converge locally
    uniformly on $\mathbf{N}$ to a continuous function $U:\mathbf{N}\to \R$ which is the unique
    solution of: for any $k=1,..,N$ 
    $$\Fk(x,t,U,DU)=0\quad\text{on }\Man{k}\, ,$$
    with all the transmission conditions \eqref{transm-netw}.
\end{theorem}

\begin{proof}
    We just sketch it, the details being tedious but straightforward at this point of the book.

    \smallskip

    \noindent\textbf{(a)} The sequence $(U^\e)$ is uniformly bounded on
    $\mathbf{N}=\bigcup_{k=1}^{N}\Man{k}$: on one hand, $U^\e \geq U^\infty$ where $U^\infty$ is the
    value function obtained by dropping the term $\e^{-1}d(x)$ in the cost and, on the other hand,
    the normal controllability implies that once we start from a point in
    $\mathbf{N}$, we can stay there. So, $U^\e(x,t)\leq M\int_0^{+\infty}\exp(-C(s))\ds<\infty$.

    \smallskip

    \noindent\textbf{(b)} It is clear that $\underline U = \limiinf U^\e = +\infty$ in $\Man{N+1}$
    but, by the classical stability result, 
    $$ \F(x,t,\underline U,D\underline U)\geq 0\quad \hbox{on}\quad \mathbf{N}$$
    and this inequality reduces to 
    $$\Fk(x,t,\underline U,D\underline U)=0\quad\text{on}\quad\Man{k}$$
    for any $k$ as above. Here we used that on each (flat) connected component of $\Man{k}$, if
    $\underline U-\phi$ has a minimum point at $(x,t)\in \Man{k}$, where $\phi$ is a smooth
    function, then $\underline U-\phi -p\cdot x$ has also a minimum point for any $p$ which is
    orthogonal to $\tMan{k}$ at $x$. Then the choice of $p$ as a minimum point of 
    $\F(x,t,\underline U(x,t),D\underline \phi(x,t)+p)$ gives the answer.

    \smallskip

    \noindent\textbf{(c)} For the $\limssup$, we just take it on $\mathbf{N}$ and using (of course)
    only the points of $\mathbf{N}$. Denoting by $\overline U$ this $\limssup$, we have,  
    $$\Fk(x,t,\overline U,D\overline U)=0\quad\text{on }\Man{k}\, .$$
    And we can also pass to the limit in the transmission conditions \eqref{transm-netw}: to do so,
    we remark that, by Lemma~\ref{lem:onlyMk}, $\overline U=\limssup (U^\e|_{\overline{\Man{k}_i}})$
    for any $k$ and $i$.

    \smallskip

    \noindent\textbf{(d)} Using the arguments of Lemma~\ref{RSub-1}, the transmission conditions
    \eqref{transm-netw} allow to show that $\overline U$ is $\Man{k}_i$-regular at each point of
    $\partial\Man{k}_i$.

    \smallskip

    \noindent\textbf{(e)} This last point allows to copy exactly the stratified proof which
    provides the key inequality $\overline U \leq \underline U$ on $\mathbf{N}$ and
    the continuity/uniqueness of $U:=\overline U \leq \underline U$.
\end{proof}

\section{Some examples}

\subsection*{An easy one}

We begin with a very easy but relevant infinite chessboard example in $\R^2$ (but this can easily be
generalized to $\R^N$): 
$$ \tMan{0}=\Z^2\; ,\; \tMan{1}= (\R\times \Z \cup \Z \times \R) \setminus \Z^2\;,$$
and $\tMan{2}=\R^2\setminus (\tMan{1}\cup \tMan{0})$.

On this stratification, one can imagine lots of control problems by imposing a
certain limitation of speed and a certain cost on each edge $E_{i,j}^-=((i,j),(i+1,j))$ or
$E_{i,j}^+=((i,j),(i,j+1))$. For instance, if $x\in E_{i,j}^-$ or $x\in E_{i,j}^+$
$$ \BCL_1(x,t):=
\begin{cases}
\{((b^x,-1),0, |b^x|/2), |b^x|\leq 2\} & \hbox{if $i$ or $j$ is a prime}\\
\{((b^x,-1),0, 2|b^x|), |b^x|\leq 1\} & \hbox{otherwise}\; ,
\end{cases}
$$
the ``$1$'' in $\BCL_1$ referring to $\tMan{1}$. And if one insists on defining $\BCL(x,t)$ on
$\tMan{2}\times (0,+\infty)$, we can always choose $\{((b^x,-1),0, \e^{-1}), b^x \in
\overline{B(0,0.1)}\}$ and, on $\tMan{0}\times (0,+\infty)$ (but also on $\tMan{1}\times
(0,+\infty)$), we just use the extension by upper semicontinuity.

Such example is very simple because each connected component of $\tMan{1}$ (or $\Man{2}$) is
extremely simple and we have no problem to check all the needed assumptions by using the simple form
of the $\BCL$ and the complete controllability.

\subsection*{Ad augusta, per angusta}

The second example in $\R^3$ is the case where, if $(x_1,x_2,x_3)$ are the coordinates of $x\in \R^3$
$$ \tMan{2}= \{\ x_1=-1,\ (x_2,x_3)\neq (0,0)\} \cup \{\ x_1=+1,\ (x_2,x_3)\neq (0,0)\} \; ,$$
$$  \tMan{1}= (-1,1)\times\{ (0,0)\}\; , \; \tMan{0}= \{(-1,0,0), (1,0,0)\} \; ,$$
and $\tMan{3}=\R^3 \setminus (\tMan{2} \cup \tMan{1} \cup \tMan{0})$. 

Here we just define the specific dynamic and cost $\BCL_i$ on $\tMan{i}\times (0,+\infty)$
$$ \BCL_2(x,t):=\{((b^x,-1),0, |b^x|/2), |b^x|\leq 2\}\; ,$$
$$ \BCL_1(x,t):=\{((b^x,-1),0, 2|b^x|), |b^x|\leq 1\}\; ,$$
and we can see $\BCL_3(x,t)$ as being $\{((b^x,-1),0, \e^{-1}), |b^x|\leq 2\}$. On $\tMan{0}\times
(0,+\infty)$, we do not impose any particular cost, the $\BCL$ at such points (but also elsewhere)
being computed using the upper semi-continuity of $\BCL$.

What could be interesting in such example, at the ``network level'', is to force the dynamic $X$ to
go through $\tMan{1}$, which can be done by a suitable choice of the initial data. Choose for
example
$$ u_0(x_1,x_2,x_3)= -10 x_1 + (1-x_2^2)_+\; .$$
If we look at the control problem on $(\tMan{2} \cup \tMan{1} \cup \tMan{0})\times (0,+\infty)$, it
is clear that, starting from a point $(-1,x_2,x_3)$ and staying on $\{x_1=-1\}$, we are
going to pay a final cost at least $10$. 

But if we decide to go directly to $(-1,0,0)$, to use the ``channel'' $ \tMan{1}$ and then to go to
the point $(1,1,0)$, the total cost on the long run (\ie for $t$ large enough) will be
$$ 2^{-1}(x_2^2+x_3^2)^{1/2} + 4 + 2^{-1} -10 \; .$$
The four terms represent successively the cost for joining $(-1,0,0)$, crossing the channel, going
to the minimum point $(1,1,0)$ and finally the terminal cost.  Of course, if $(x_2^2+x_3^2)^{1/2}$
is not too large, this strategy is far better than the other one.

\chapter{Further Discussions and Open Problems}
\label{chap:openpb-partVI}

\abstract{Several questions and possible extensions are discussed here. The possible
generalization of the results for KPP which echo the key question of the convergence of the
vanishing viscosity method in (even relatively simple) stratified cases. Also considered are some
open problems including jumps and networks.}

This part of the book provides several open problems and the interest of some of them we already
described. But let us make some further comments on these.

\bigskip

\noindent\textbf{(a)} We first come back on the {\em applications to KPP}, considering the convergence of the vanishing
viscosity method. While, in the case of codimension~$1$ discontinuities, rather general results are
available, \cf for example Theorem~\ref{pro:viscous}, the least generalization to even very simple
stratified situations remains open.

Two questions are really very puzzling: on one hand, is it ``always'' true that the vanishing
viscosity method converges to the maximal Ishii (sub)solution? (and does the ``always'' require
some restrictions either on the geometry of discontinuities or on the Hamiltonians?). 
On the other hand, if this first result is correct, can we characterize the maximal Ishii (sub)solution?
This second question is discussed in Section~\ref{sec:puzzling} and as the reader can notice it
on page~\pageref{pb:cross}, even the simplest cases cause problems.\index{Vanishing viscosity method!in the stratified framework}

\bigskip

\noindent\textbf{(b)} Problems with {\em jumps}, even if we made a point to have a section which is
dedicated to discuss them, are completely open in general and perhaps/probably a change of strategy
in most of the proofs is needed. Hence almost everything needs to be done in this direction.

\bigskip

\noindent\textbf{(c)} The {\em network case}---and may be even more striking with the introduction of ${\rm
PC}^1(\mathbf{N})$---rises the questions of a pure pde comparison proof in the stratified or
network setting; with the idea of considering more general Hamiltonians
like in the case of codimension~$1$ discontinuities. But first, it is not clear for us that, the way
we define it, ${\rm PC}^1(\mathbf{N})$ is the right space of test-function (or $\PC1$ defined in an
analogous way in the stratified framework). Maybe in addition to be $C^1$ on each connected
component $\Man{k}_i$ of $\Man{k}$ for any $k$ and $i$, test-functions should also have $C^1$ extensions to
$\overline{\Man{k}_i}$ like in the codimension~$1$ case.

Next, clearly the ``Magical Lemma'', Lemma~\ref{lem:comp.fundamental} is based on the idea that
either an optimal trajectory stays on $\Man{k}$ and this leads to the $\F^k$-inequality or it enters
in the domain $\Man{k+1}\cup \cdots \cup \Man{N+1}$ and in this case, ``stratified inequalities'' are
missing. Above, in the network case, we introduced the transmission conditions \eqref{transm-netw}
which partly play this role. But this is not enough because, on one hand, the above reference on the
``Magical Lemma'' means that {\em supersolution inequalities} are missing and while
\eqref{transm-netw}---which are just here to ensure the regularity of subsolutions---are ``poor''
replacements for the needed inequalities since we have to take into account ALL the inner dynamics
to $\Man{k+1}\cup \cdots \cup \Man{N+1}$. 

We also refer the reader to the notion of \FL-solutions: in the stratified setting, the missing
inequalities are the analogue of the $H_1^+,H_2^-$-ones for the \FL-supersolutions and we recall
that these inequalities are automatically satisfied by the subsolutions thanks to
Proposition~\ref{sub-up-to-b}.

Yet a very formal proof is easy to write (with other open questions there if the Hamiltonians are
not convex!): 
\begin{enumerate}
\item[1.] We introduce the space $\PC1$ of continuous functions with are $C^1$ on each $\Man{k}$.
\item[2.] We apply the localization techniques which transform the subsolution into a coercive,
    strict subsolution.  
\item[3.] The regularization of a strict subsolution gives a strict subsolution which is in $\PC1$,
    hence this subsolution becomes a test-function (not completely true, even in the convex case,
        but very close to be valid. A real difficulty in the non-convex case).
\item[4.] Since the strict subsolution is a test-function, one should coclude by applying the
    definition (here clearly the problem is with the definition, \ie with the missing correct
        inequalities).
\end{enumerate}

Obviously a lot of works and certainly new ideas are needed to make this formal proof work!

%
%

\chapter*{Final words}
\addcontentsline{toc}{part}{Final words}

If we were to choose some concluding thoughts for the reader to be left with, we would
certainly insist on this one: even if the general framework we developed has some complexities and
technicalities, in particular for the stratified approach, 

\begin{center}
    \begin{minipage}{0.8\textwidth}\emph{we are convinced that the tandem \emph{normal
        controllability--tangential continuity} is the right setting for producing general results.}
    \end{minipage}
\end{center}

\noindent And we believe we have clearly justified this claim in the introduction, as well as in 
various places throughout the manuscript. 

Clearly, this leaves aside a lot of interesting examples where these assumptions are not fully
satisfied, at least partially. Probably some of them can be treated by some \emph{ad hoc}
modifications of our approach, but certainly other ones require different treatments.

This also leads us to insist again on some very interesting problems and open questions on the subject,
the resolution of which can lead to real improvements in what we are doing here. Some of them are
presented in the \emph{Further Comments} sections; let us recall here those which we find are the
most puzzling.
\begin{enumerate}
    \item The most iconic, simple problem we still do not know how to solve is the
        \emph{cross problem}, presented in Section~\ref{sec:puzzling}: in such configurations, the
        stratified approach works very well and provides a unique stratified solution, the minimal
        Ishii solution. But five years after, through coffee and headaches spending
        hours\footnote{To be frank, we have probably spent altogether something like a month or two
        thinking about this question!} on this, we are still not able to define and work with a
        maximal viscosity (sub)solution in this setting. More generally the question of whether it
        is possible to \emph{always} identify the maximal Ishii subsolution in the stratified
        framework is largely open. In particular, the case where the discontinuity is just a line in
        $\R^3$ is also very challenging.

    \item Assuming this first question is solved, the next one concerns the convergence of the
        vanishing viscosity method: is it \emph{always} true, at least in the stratified framework,
        that the vanishing viscosity method converges to the maximal Ishii subsolution? Or are there
        different characterizations depending on the nature of the discontinuities?
        \index{Vanishing viscosity method!general questions on}

    \item In our framework, the required notion of stratification is \TFS. This is clearly a
        restriction compared to a general stratification. Is this restriction really necessary to
        obtain comparison results? Or is it a condition linked to our method of proof? We have no
        idea about the answer.
        
    \item Because it is a natural notion of solution in the control/convex framework, it seems to us that the stratified
    Barron-Jensen approach should be pushed. We refer the reader to the end of Section~\ref{sect:SBJ} for some
    possible starting points of future investigations.

    \item A question that people keep asking and which is indeed very puzzling: is it possible to
        have a pure pde proof for the stratified problem, \ie some kind of generalization of the
        ``network approach'' in the stratified context? Of course, this would open the way to the
        treatment of non-convex equations, \emph{what a dream!}
        
    \item We have just scratched the surface for the theory and possible applications of the stratified approach for
    state-constraints: it is hard to imagine the wide scope of all problems which can be addressed through this approach.
    
\end{enumerate}

\

Let us end here by saying that, though we tried to minimize them, we take all the imperfections and
maybe mistakes in this book as good news. First, because embracing them is always more productive
(and more relaxing than being angry about it!), but also because anything that can stimulate further
research on the subject is positive. 

We have no doubt that new results and better understanding will unfold. In that process, we hope
that this book or at least some specific parts in it will serve as ``building blocks" for later
applications and new developments.


\appendix 

\part*{Appendices}
\addcontentsline{toc}{part}{Appendices}
\fancyhead[CO]{Appendices}

\chapter{Notations and Terminology}
\label{app:notations}

\begin{tabular}{ll}
$\F,\G,H$ & Generic Hamiltonians\\
$\mathcal{A,E}$ & Generic subsets of $\R^N$\\
$\mathcal{O,F}$ & Generic open and closed subsets of $\R^N$\\
$\mathcal{K}$ & Generic compact subset of $\R^N$\\
$B(y,r)$ & Open ball of center $y\in\R^k$ and of radius $r>0$
for the Euclidian norm.\\
\ & \ 
\\
$A$ & a compact, convex subset of $\R^p$\\
$\mathbb{A}$ & the space of controls, $\mathbb{A}=L^\infty(0,T;A)$ \\
$\BCL(\cdot,\cdot)$ & set-valued map combining all the dynamics, costs, discount factors, p.\pageref{not:BCL}\\
$(X,T,D,L)$ & a generic trajectory of the differential inclusion, p.\pageref{not:XTDL}\\
$\mT(x,t)$ & space of controled trajectories such that $(X,T,D,L)(0)=(x,t,0,0)$, p.\pageref{not:mT}\\
$\mT^\reg(x,t)$ & space of regular controlled trajectories such that \\
& $(X,T,D,L)(0)=(x,t,0,0)$, p.\pageref{not:mTreg}\\

\\

$s_-,s_+$ & negative and positive parts of $s\in \R$, p.\pageref{pos.neg}\\
$z_*,z^*$ & lower and upper semi-continuous enveloppes of a function $z$, p.\pageref{not:semic}\\
$\usc,\lsc$ & upper/lower semi-continuous function, p.\pageref{not:usc}\\
$V_k$ & $k$-dimensional vectorial subspace, typically $V_k=\R^k\times\{0\}^{N-k}$, p.\pageref{not:Vk}\\
$Q^{x,t}_{r,h}$ & the open cylinder $B(x,r)\times (t-h,t)$.\\
$Q^{x,t}_{r,h}[\mF]$ & the open cylinder $\big(B(x,r)\cap\mF\big)\times (t-h,t)$, p.\pageref{not:cyl}.\\
\\

$\USCS(\mF)$ & set of \usc subsolutions on $\mF$, p.\pageref{not:USCS}\\
$\LSCS(\mF)$ & set of \lsc supersolutions on $\mF$, p.\pageref{not:LSCS}\\
$\PC1$ & piecewise $C^1$-smooth test functions, p.\pageref{not:PCun}\\

\ & \
\\

\end{tabular}

\begin{tabular}{ll}
$\M$ & a general stratification of $\R^N$ or $\R^N\times (0,\Tf)$, p.\pageref{page:HSTreg}\\
\AFS & Admissible Flat Stratification, p.\pageref{not:AFS}\\
\LFS & Locally Flattenable Stratification, p.\pageref{page:HSTreg}\\
\TFS & Tangentially Flattenable Stratification, p.\pageref{page:HSTtfs}\\

\\

\TC & Tangential Continuity (pde version), p.\pageref{not:TC}\\
\TCBCL & Tangential Continuity (control version), p.\pageref{page:TCBCL}\\
\NCw & Weak Normal Controllability, p.\pageref{page:wNC}\\
\NCe & Normal Controllability (pde version), p.\pageref{not:NC}\\
\NCBCL & Normal Controllability (control version, p.\pageref{page:NCBCL}\\
\Mong & Monotonicity Assumption, p.\pageref{not:Mong}\\

\\

$\SCR$ & Strong Comparison Result, p.\pageref{not:SCR},\pageref{not:SCR2}\\
$\LCR$ & Local Comparison Result, p.\pageref{not:LCR}\\
$\GCR$ & Global Comparison Result, p.\pageref{not:GCR}\\

\end{tabular}

\

\

\textbf{Notions of solutions} (see also appendix B for quick reference)

\

\begin{tabular}{ll} 
    (CVS) & Ishii solutions / Classical viscosity solutions, 
    p.\pageref{eq:ishii.cond} (see also Section~\ref{subsec.disc.visc.sol})\\
    (FLS) & Flux-Limited Solution, p.\pageref{defiFL}\\
    (JVS) & Junction Viscosity solution, p.\pageref{defiJVS}\\
    (S-Sub) & Stratified subsolutions, p.\pageref{defiStrat}\\
    (S-Super) & Stratified supersolutions, p.\pageref{defiStrat}\\

\end{tabular}

\

\textbf{NB: }The ``good framework for HJ Equations with discontinuities'' is defined
p.\pageref{page:GF.HJE}.

\chapter{Assumptions, Hypotheses, Notions of Solutions}
\label{app:hyp}    

The page number refers to the page where the assumption is stated for the first time in the book.

\section*{Basic or fundamental assumptions}

\begin{app}{\HBACPc}{Basic Assumptions on the Control Problem -- Classical
    case,}{p.~\pageref{page:HBACPc}}
\begin{enumerate} 
\item[$(i)$] The function $\u0: \R^N\to \R$ is a bounded, uniformly continuous function.
\item[$(ii)$] The functions $b,c,l$ are bounded, uniformly continuous on $\R^N \times [0,\Tf] \times A$.
\item[$(iii)$] There exists a constant $C_1>0$ such that, for any $x,y \in \R^N$, $t \in [0,\Tf]$, $\alpha \in A$, we have
$$ |b(x,t, \alpha)-b(y,t,\alpha)| \leq C_1 |x-y|\; .$$
\end{enumerate} 
\end{app}

\begin{app}{\HBACP}{Basic Assumptions on the Control Problem,}{p.~\pageref{page:HBACP}}
\begin{enumerate} 
\item[$(i)$] The function $\u0: \R^N\to \R$ is a bounded, continuous function.
\item[$(ii)$] The functions $b,c,l$ are bounded, continuous functions on $\R^N \times [0,\Tf] \times A$ and the sets $(b,c,l)(x,t, A)$ are convex compact subsets of $\R^{N+2}$ for any $x\in \R^N$, $t\in [0,\Tf]$ \footnote{The last part of this assumption which is not a loss of generality will be used for the connections with the approach by differential inclusions.}.
\item[$(iii)$] For any ball $B\subset \R^N$, there exists a constant $C_1 (B)>0$ such that, for any $x,y \in \R^N$, $t \in [0,\Tf]$, $\alpha \in A$, we have
$$ |b(x,t, \alpha)-b(y,s,\alpha)| \leq C_1(B)\left( |x-y| +|t-s| \right)\; .$$
\end{enumerate} 
\end{app}

\begin{app}{\HBAHJ}{Basic Assumptions on the Hamilton-Jacobi equation,}{p.~\pageref{page:BA-HJ}}\\[2mm]
There exists a constant $C_2>0$ and, for any ball $B\subset \R^N\times [0,\Tf]$, for any $R>0$,
    there exists constants $C_1=C_1(B,R)>0, \gamma (R)\in \R $ and a modulus of continuity $m=m(B,R) : [0,+\infty) \to [0,+\infty)$ such that, for any $x,y \in B$, $t,s \in [0,\Tf]$, $-R \leq r_1 \leq r_2 \leq R$ and $p,q \in \R^N$
$$ |H(x,t,r_1,p)-H(y,s,r_1,p)|\leq C_1[|x-y|+|t-s|]|p| + m(|x-y|+|t-s|)\; ,$$
$$ |H(x,t,r_1,p)-H(x,t,r_1,q)|\leq C_2 |p-q|\; ,$$
$$ H(x,t,r_2,p)-H(x,t,r_1,p) \geq \gamma(R) (r_2-r_1) \; .$$
\end{app}

\begin{app}{\HGCP}{Basic Assumption on the $p_t$-dependence,}{p.~\pageref{page:HGCP}}\\[2mm] 
For any $(x,t,r,p_x,p_t)\in \mF\times(0,\Tf]\times\R\times\R^N\times\R$,
the function $p_t\mapsto \G\big(x,t,r,(p_x ,p_t)\big)$ is increasing and $\G\big(x,t,r,(p_x ,p_t) \big)
\to +\infty$ as $p_t \to +\infty$, uniformly for bounded $x,t,r,p_x$.
\end{app}

\begin{app}{\HBAConv}{Basic Assumption in the convex case,}{p.~\pageref{page:HJConv}}\\[2mm]
	$H(x,t,r,p)$ is a locally Lipschitz function which is convex in $(r,p)$. 
	Moreover, for any ball $B\subset \R^N\times [0,\Tf]$, for any $R>0$, there exists constants $L=L(B,R), K=K (B,R) >0$ 
	and a function $G=G(B,R):\R^N\to [1,+\infty[$ such that, for any $x,y \in B$, $t,s \in [0,\Tf]$, $-R \leq u \leq v \leq R$ and $p \in \R^N$
	$$ D_p H(x,t,r,p)\cdot p -H(x,t,u,p) \geq G(p) -L\; ,$$
	$$ |D_x H(x,t,r,p)|, |D_t H(x,t,r,p)| \leq K G(p)( 1+ |p|)\; ,$$
	$$ D_r H(x,t,r,p) \geq 0 \; .$$
\end{app}

\section*{Stratification assumptions}

\begin{app}{\HSTGEN}{General Stratifications,}{p.~\pageref{page:HSTgen}}\\[2mm]
        $\M=(\Man{k})_{k=0..N}$ is a \emph{General Stratification} of $\R^N$
        if the following set of hypotheses \HSTGEN is
        satisfied:
        \begin{enumerate}
    \item[$(i)$] For any $k=0..N$, $\Man{k}$ is a $k$-dimensional submanifold of $\R^N$.
    \item[$(ii)$] If $\Man{k}_i \cap \overline {\Man{l}_j}\neq \emptyset$ for some $l>k$
       then $\Man{k}_i \subset \overline {\Man{l}_j}$.
    \item[$(iii)$] For any $k=0..N$, $\overline{\Man{k}} \subset
       \Man{0}\cup\Man{1}\cup\cdots\cup\Man{k}$. 
    \item[$(iv)$] If $x \in \Man{k}$ for some $k=0..N$, there exists $r=r_x>0$ such that
    \begin{enumerate}
    \item[$(a)$] $B(x,r) \cap \Man{k}$ is a {\em connected} submanifold of $\R^N$;
     \item[$(b)$]  For any $l<k$, $ B(x,r) \cap \Man{l}=\emptyset$\,;
     \item[$(c)$]  For any $l>k$, $B(x,r)\cap\Man{l}$ is either empty or has at most a finite
                number of connected components\,;
    \item[$(d)$] For any $l>k$, $B(x,r)\cap\Man{l}_j\neq\emptyset$ if and only if
                $x\in\partial\Man{l}_j$.
        \end{enumerate}
    \end{enumerate}
\end{app}

\begin{app}{\HSTFLAT}{Flat Stratifications,}{p.~\pageref{page:HSTflat}}\\[2mm]
    The stratification $\M$ is an \AFS if it satisfies \HSTGEN, with the exception of 
    property \HSTGEN-$(iv)(a)$, which is replaced by 
    \begin{equation*} 
        \text{\HSTFLAT-}(iv)(a)\quad B(x,r) \cap \Man{k} = B(x,r) \cap (x+V_k) \text{ for some
        }(x+V_k)\in\V{k}(x)\;.
    \end{equation*}
    We denote by \HSTFLAT the set of conditions $(i)-(iv)$ with this replacement.
\end{app}

\begin{app}{\HSTLFS}{Locally Flattenable Stratifications,}{p.~\pageref{page:HSTlfs}}\\[2mm]
    $\M=(\Man{k})_{k=0..N}$ is a \emph{locally flattenable stratification} of
    $\R^N$---\,\LFS in short---\,if it satisfies the two following assumptions denoted by \HSTLFS
    \begin{enumerate}
        \item[$(i)$] the following decomposition holds:
            $\R^N=\Man{0}\cup\Man{1}\cup\cdots\cup\Man{N}$; 
        \item[$(ii)$] for any $x \in \R^N$, there exists $r=r(x)>0$ and a $C^{1,1}$-change of
            coordinates $\Psi^x: B(x,r) \to \R^N$ such that $\Psi^x(x)=x$ and
            $\{\Psi^x(\Man{k}\cap B(x,r))\}_{k=0..N}$ is the restriction to $\Psi^x(B(x,r))$ of an
            \AFS in $\R^N$.
    \end{enumerate}
\end{app}

\begin{app}{\HSTTFS}{Tangentially Flattenable Stratifications,}{p.~\pageref{page:HSTtfs}}\\[2mm]
    We say that $\M=(\Man{k})_{k=0..N}$ is a \emph{Tangentially Flattenable Stratification of $\OO$}\,---\TFS
    in short---\,if the following hypotheses hold:
    \begin{enumerate}
        \item[$(i)$] Hypotheses \HSTGEN are satisfied;
        \item[$(ii)$] for any $k$, $\Man{k}$ is a $C^{1,1}$-submanifold of $\OO$; moreover if $x\in \Man{k}$,
            there exists $r=r_x>0$ such that $B(x,r)\subset \OO$ and a $C^{1,1}$-diffeomorphism $\Psi_x$
            defined on $B(x,r)$ such that $\Psi_x(x)=x$ and
            $$\Psi_x(B(x,r) \cap \Man{k})=\Psi_x(B(x,r)) \cap (x+V_k)$$ 
            where $V_k$ is a $k$-dimensional vector subspace of $\R^N$;
        \item[$(iii)$] setting $\tMan{l}:=\Psi_x(B(x,r) \cap \Man{l}$ and
            $\tMan{l}_j= \Psi_x(B(x,r) \cap \Man{l}_j)$ for any connected component $\Man{l}_j$ of
            $\Man{l}$, 
            \begin{enumerate}
                \item[$(a)$] for any $l<k$, $ \tMan{l}=\emptyset$\,;
                \item[$(b)$] for any $l>k$, $\tMan{l}$ is either empty or has at most a finite
                number of connected components;
            \item[$(c)$] if $x\in\partial\tMan{l}_j$ and $y\in \tMan{l}_j$,
                $\Psi_x(B(x,r)) \cap (y+V_k) \subset \tMan{l}_j\;.$
            \end{enumerate}
    \end{enumerate}
\end{app}

\

\noindent\textbf{N.B.} \emph{Convergence in the sense of stratifications is
        defined} pp.~\pageref{page:strat.conv.un} and \pageref{page:strat.conv.deux}. 

\section*{Assumptions for the differential inclusion and the value function}

\begin{app}{\HBCLa}{Fundamental assumptions on the set-valued map $\BCL$,}{p.\pageref{page:HBCLa}}\\[2mm]
The set-valued map $\BCL:\R^N\times[0,\Tf]\to\mathcal{P}(\R^{N+3})$ satisfies
\begin{enumerate}
\item[$(i)$] The map $(x,t)\mapsto\BCL(x,t)$ has compact, convex images
      and is upper semi-continuous;
  \item[$(ii)$] There exists $M>0$, such that for any $x\in\R^N$ and $t> 0$,
      $$\BCL(x,t)\subset \big\{(b,c,l)\in\R^{N+1}\times\R\times\R:|b|\leq M; |c|\leq M; | l | \leq
          M\big\}\;,$$
  \end{enumerate}
\end{app}

\begin{app}{\HBCLb}{Structure assumptions on the set-valued map $\BCL$,}{p.\pageref{page:HBCLb}}\\[2mm]
There exists $\uc, K >0$ such that
\begin{enumerate}
\item[$(i)$] For all $x\in \R^N$, $t\in [0,\Tf]$ and $b=(b^x,b^t) \in \B (x,t)$, $-1 \leq b^t \leq 0$. Moreover, there exists $b=(b^x,b^t) \in \B (x,t)$ such that $b^t =-1$.
\item[$(ii)$] For all $x\in \R^N$, $t\in [0,\Tf]$,  if $((b^x,b^t),c,l) \in \BCL (x,t)$, then $-Kb^t + c \geq 0$. 
\item[$(iii)$] For any $x \in \R^N$, there exists an element in $\BCL(x,0)$ of the form $((0,0),c,l)$ with $c \geq \uc$.
\item[$(iv)$] For all $x\in \R^N$, $t\in [0,\Tf]$, if $(b,c,l) \in \BCL (x,t)$ then $\max(-b^t,c,l)\geq \uc$.
\end{enumerate}
\end{app}

\

\noindent \HBCL is just the conjunction of \HBCLa and \HBCLb.

\

\noindent \hyp{\hU} --- The value function $U$ is locally bounded on $\Omegb \times [0,\Tf]$,~p.\pageref{page:hU}.

\

\section*{Normal controllability, tangential continuity, monotonicity}

\begin{app}{\TC}{Tangential Continuity, HJ version,}{p.~\pageref{page:TC}}\\[2mm]
For any $\vc_1=(\vt_1,\vn), \vc_2=(\vt_2,\vn) \in B_\infty (\vcb,r)$, $|u|\leq R$ and $p\in \R^{N}$,
    then 
    $$|\G(\vc_1,u,p)- \G(\vc_2,u,p)|\leq C^R_1|\vt_1-\vt_2|.|p| +m^R\big(|\vt_1-\vt_2|\big)\;.$$
\end{app}

\begin{app}{\TCs}{Strong Tangential Continuity,}{p.~\pageref{page:TCs}}\\[2mm]
For any $R>0$, there exists $C^R_1>0$ and a modulus of continuity $m^R :
    [0,+\infty[ \to [0,+\infty[$ such that for any $\vc_1=(\vt_1,\vn), \vc_2=(\vt_2,\vn) \in
    B_\infty (\vcb,r)$, $|u|\leq R$, $p=(p_\vt,p_\vn)\in \R^{N}$, then 
    $$|\G(\vc_1,u,p)- \G(\vc_2,u,p)|\leq C^R_1|\vt_1-\vt_2|.|p_\vt| +m^R\big(|\vt_1-\vt_2|\big)\;.$$
\end{app}

\begin{app}{\NCw}{Weak Normal Controllability,}{p.~\pageref{page:wNC}}
\begin{enumerate}
\item[(i)] If $N-k>1$, there exists $e \in \R^{N-k}$ such that, for any $R>0$, we have
$$
\G(\vc,u,(p_\vt,Ce))\to +\infty \quad\hbox{when  }C \to +\infty\;,
$$
uniformly for $\vc=(\vt,\vn)\in B_\infty (\vcb,r)$, $|u|\leq R$, $|p_\vt|\leq R$. 
\item[(ii)] If $N-k=1$, this property holds for $e =+1$.
\item[(iii)] If $N-k=1$, this property holds for $e =-1$.
\end{enumerate}
\end{app}

\begin{app}{\NCe}{Normal Controllability, HJ version,}{p.~\pageref{page:NC}}\\[2mm]
    For any $\vc=(\vt,\vn)\in B_\infty (\vcb,r)$, $|u|\leq R$, $p=(p_\vt,p_\vn)\in \R^{N}$, then 
    $$\G(\vc,u,p)\geq C^R_2 |p_\vn| - C^R_3|p_\vt| -C^R_4\;.$$
\end{app}

\begin{app}{\NCoH}{Normal Controllability -- codimension $1$ case,}{p.~\pageref{page:NCoH}} \\[2mm]
For any $(x,t) \in \H \times [0,\Tf]$, there exists $\delta =
    \delta(x,t)$ and a neighborhood $\VV = \VV(x,t)$ such that, for any $(y,s) \in \VV$ \\[-4mm]
    $$\begin{aligned}[]
    [-\delta,\delta\,] & \subset \{b_1(y,s,\alpha_1)\cdot e_N,\ \alpha_1 \in A_1\}\quad 
    \hbox{if }(y,s)\in \Omegb_1\;,\\
    [-\delta,\delta\,] & \subset \{b_2(y,s,\alpha_2)\cdot e_N,\ \alpha_2 \in A_2\}\quad 
    \hbox{if }(y,s)\in \Omegb_2\;,
    \end{aligned}$$
    where $e_N=(0,0\cdots,0,1)\in \R^N$.
\end{app}

\begin{app}{\Mong}{Monotonicity property,}{p.~\pageref{page:Mong}}\\[2mm]
    For any $R>0$, there exists $\lambda_R, \mu_R \in \R$, such that one of the two following properties holds
\begin{enumerate}
    \item[] \Monu : $\lambda_R>0$ and for any $\vc\in B_\infty (\vcb,r)$, $p=(p_\vt,p_\vn)\in \R^{N}$,
        any $ R\leq u_1 \leq u_2 \leq R$, 
        \begin{equation}\label{mon-u-A}
        \G(\vc,u_2,p)-\G(\vc,u_1,p)\geq \lambda^R(u_2-u_1)\;; 
        \end{equation}
    \item[] \Monp :
         \eqref{mon-u-A} holds with $\lambda_R=0$, we have $\mu_R>0$ and
        \begin{equation}\label{mon-p-A}
            \G(\vc,u_1,q)-\G(\vc,u_1,p)\geq \mu^R(q_{\vt_1}-p_{\vt_1})\; ,
        \end{equation}
        for any $q=(q_\vt,p_\vn)$ with $p_{\vt_1}\leq q_{\vt_1}$ and
        $p_{\vt_i}=q_{\vt_i}$ for $i=2,...,p$.
\end{enumerate} 
\end{app}

\begin{app}{\TCBCL}{Tangential Continuity -- $\BCL$ version (when $\M$ is a \TFS and we can assume \wlg that
    $\Man{k}=(x,t) + V_k$).}{p.\pageref{page:TCBCL}}\\[2mm]
For any  $0\leq k\leq N+1$ and for any $(x,t) \in
\Man{k}$, there exists a constant $C_1>0$ and a modulus $m:[0,+\infty)\to \R^+$ such that,
for any $j\geq  k$, if
$(y_1,t_1),(y_2,t_2)\in\Man{j}\cap B((x,t),r)$ with 
$(y_1,t_1)-(y_2,t_2) \in V_k$, then for any $(b_1,c_1,l_1) \in \BCL(y_1,t_1)$,
there exists $(b_2,c_2,l_2) \in \BCL(y_2,t_2)$ such that
$$ |b_1-b_2|\leq C_1(|y_1-y_2|+|t_1-t_2|)\quad,\quad 
|c_1-c_2|+|l_1-l_2| \leq m\big(|y_1-y_2|+|t_1-t_2|\big)\; .$$
\end{app}

\begin{app}{\NCBCL}{Normal Controllability -- $\BCL$ version (when $\M$ is a \TFS and we can assume \wlg that
    $\Man{k}=(x,t) + V_k$).}{p.\pageref{page:NCBCL}}\\[2mm]
For any  $0\leq k\leq N+1$ and for any $(x,t) \in
\Man{k}$, there exists $\delta = \delta(x,t)>0$, such that, for any $(y,s) \in B((x,t),r)$,
one has $$B(0,\delta) \cap V_k^\bot \subset P^\bot\left(\B(y,s) \right)\, .$$
\end{app}

\section*{Localization, convexity, subsolutions}

\begin{app}{\LOCa}{localization hypothesis 1,}{p.~\pageref{page:LOCa}}\\[2mm]
If $\mF$ is unbounded, for any $u \in \USCS(\mF)$, for any $v \in \LSCS(\mF)$, there exists a
    sequence $ (u_\alpha)_{\alpha>0}$ of \usc subsolutions of \eqref{GenHJ} such that $u_\alpha(x) -
    v (x) \to -\infty$ when $|x|\to +\infty$, $x\in \mF$.  Moreover, for any $x\in \mF$,
    $u_\alpha(x) \to u(x)$ when $\alpha \to 0$.\end{app}

\begin{app}{\LOCb}{localization hypothesis 2,}{p.~\pageref{page:LOCb}}\\[2mm]
    For any $x \in \mF$, $r>0$, if $u \in \USCS(\mF^{x,r})$, there exists a sequence
    $(u^\delta)_{\delta>0}$ of functions in $\USCS(\mF^{x,r})$ such that $\floorF{u^\delta -
    u}\geq\eta(\delta)>0$ for any $\delta$.  Moreover, for any $y\in \mF^{x,r}$, $u^\delta(y) \to
    u(y)$ when $\delta \to 0$.
\end{app}

\noindent\textbf{N.B.} we recall that $\displaystyle
\floorF{f}:=f(x)-\max_{y\in\partial\mF^{x,r}}f(y)\;.$

\begin{app}{\LOCaEV}{Localization assumption one -- evolution case,}{p.~\pageref{page:LOCevol1}}\\[2mm]
If $\mF$ is unbounded, for any $u \in \USCS(\mF\times [0,\Tf])$, for any $v \in \LSCS(\mF\times
    [0,\Tf])$, there exists a sequence $(u_\alpha)_{\alpha>0}$ of \usc subsolutions of \eqref{GenHJ-evol}
    such that $u_\alpha(x,t) - v (x,t) \to -\infty$ when $|x|\to +\infty$, $x\in \mF$, uniformly for
    $t\in [0,\Tf]$. Moreover, for any $x\in \mF$, $u_\alpha(x,t) \to u(x,t)$ when $\alpha \to 0$.
\end{app}

\begin{app}{\LOCbEV}{Localization assumption two -- evolution case,}{p.~\pageref{page:LOCevol2}}\\[2mm]
     For any $x \in \mF$, if $u \in \USCS(Q^{x,t}_{\bar r,\bar h}[\mF])$ for some $0<\bar r$, $0<\bar h <t$, there
    exists $0<h\leq \bar h$ and a sequence $(u^\delta)_{\delta>0}$ of functions in
    $\USCS(Q^{x,t}_{\bar r,h}[\mF])$ such that $\floorFt{u^\delta-u}\geq\tilde\eta(\delta)>0$ with
    $\tilde\eta(\delta)\to0$ as $\delta\to0$.  Moreover $u^\delta \to u$ uniformly on $\overline{Q^{x,t}_{r,h}[\mF]}$
    when $\delta \to 0$.\end{app}

\begin{app}{\HSubHJ}{Existence of a subsolution,}{p.~\pageref{page:HJSub}}\\[2mm]
	There exists a $C^1$-function $\psi : \R^N \times [0,\Tf] \to \R$ which is a subsolution of \eqref{HJ:evol} 
	and which satisfies $\psi(x,t) \to -\infty$ as $|x| \to +\infty$, uniformly for $t\in [0,\Tf]$ and $\psi (x,0) \leq u_0(x) $ in $\R^N$.
\end{app}

\begin{app}{\HConv}{Convexity for a general Hamiltonian,}{p.~\pageref{page:HConv}}\\[2mm] 
    For any $x\in B_\infty (\vcb,r)$, the function $(u,p)\mapsto G(\vc,u,p)$ is convex.
\end{app}

\begin{app}{\QCR}{Quasiconvex Hamiltonians in $\R$,}{p.~\pageref{page:QCR}}\\[2mm]
$f:\R \to \R$ is continuous, coercive and quasi-convex, i.e for any $a\in\R$, the lower level set $\{x:f(x)\leq a\}$ is convex or
equivalently, for any $x,y\in \R$ and $\lambda\in(0,1)$, $$f(\lambda x+(1-\lambda)y)\leq
\max\{f(x),f(y)\}\;.$$
\end{app}

\begin{app}{\HQC}{Quasi-convex Hamiltonians in the $e_N$-direction in
    $\R^N$,}{p.~\pageref{page:HQC}}\\[2mm]
For any fixed $(x,t,r,p')$, the function $h:s\mapsto H(x,t,r,p'+se_N)$ satisfies \QCR. As a consequence, $H=\max(H^+,H^-)$
where $s\mapsto H^+(x,t,r,p'+p_N e_N)$ is decreasing and $s\mapsto H^-(x,t,r,p'+p_N e_N)$ is increasing.
\end{app}

\section*{Comparison results}

\begin{app}{$\GCR^\mF$}{Global Comparison Result in $\mF$;}{p.~\pageref{page:GCR}}\\[2mm]
	 For any $u \in \USCS(\mF)$, for any $v \in \LSCS(\mF)$, we have $u\leq v$ on $\mF$\;.
\end{app}

\begin{app}{$\LCR^\mF$}{Local Comparison Result in $\mF$,}{p.~\pageref{page:LCR}}\\[2mm]
 For any $x \in \mF$, there exists $\bar r>0$ such that, 
if $u \in \USCS(\mF^{x,\bar r})$, $v \in \LSCS(\mF^{x,\bar r})$ then for any $0<r\leq \bar r$,
$$ \max_{\overline{\mF^{x,r}}}(u-v)_+ \leq \max_{\partial \mF^{x,r} }(u-v)_+ .$$
\end{app}

\begin{app}{\LCREV}{Local comparison result -- evolution case,}{p.~\pageref{page:LCRevol}}\\[2mm]
    For any $(x,t) \in \mF \times (0,\Tf]$, there exists $\bar r>0$, $0<\bar h<t$ such that, for any $0<r\leq \bar r$, $0<h<\bar h$, 
    if $u \in \USCS(Q^{x,t}_{\bar r,\bar h}[\mF])$, $v \in \LSCS(Q^{x,t}_{\bar r,\bar h}[\mF])$,
    $$ \max_{\overline{Q^{x,t}_{r,h}[\mF]}}(u-v)_+ \leq \max_{\partial_p
    Q^{x,t}_{r,h}[\mF]}(u-v)_+\;.$$
\end{app}

    \noindent \textbf{N.B.:} here, $\partial_p Q^{x,t}_{r,h}[\mF]$ stands for the parabolic boundary of $Q^{x,t}_{r,h}[\mF]$, 
    composed of a ``lateral'' part and an ``initial'' part as follows
    $$\begin{aligned}\partial_p Q^{x,t}_{r,h}[\mF] &= \Big\{(\partial {B(x,r)}\cap \mF)\times
    [t-h,t]\Big\}\bigcup \Big\{(\overline{B(x,r)}\cap \mF)\times \{t-h\}\Big\}\\
    &=: \partial_\mathrm{\,lat} Q \cup \partial_\mathrm{\,ini}Q\;.\end{aligned}
    $$

\begin{app}{\LCRxt}{Local Comparison Result around $(x,t)$ in the stratified
    case,}{p.~\pageref{page:LCRxt}}\\[2mm]
            There exists $r=r(\xb,\tb)>0$ and $h=h(\xb,\tb)\in(0,\tb)$ such
            that, if $u$ and $v$ are respectively a strict stratified
            subsolution and a stratified supersolution of some
            $\psi$--Equation in $Q^{\xb,\tb}_{r,h}$ and if $\displaystyle
            \max_{\overline{Q^{\xb,\tb}_{r,h}}}(u-v) >0$, then $$
            \max_{\overline{Q^{\xb,\tb}_{r,h}}}(u-v) \leq \max_{\partial_p
                Q^{\xb,\tb}_{r,h}}(u-v)\;.$$
            \textbf{N.B.} here, $\psi$-equation means an equation with obstacle
            $\psi$, a continuous function: $\max(\F(x,t,w,Dw),w-\psi)=0$.
\end{app}

\section*{Notions of solutions}

\textbf{N.B.} The following definitions are just gathered here as a quick
reminder, the reader will find more details and the precise definition
on the page given in reference.

\begin{app}{(CVS)}{Classical Ishii Solutions for the hyperplane case,}{p.~\pageref{eq:ishii.cond}}\\[2mm]
            This is the ``classical'' notion of viscosity solution (hence the
            acronym \CVS) where on the hyperplane the relaxed condition reads (in
            the viscosity sense)
        \begin{equation*}\begin{cases}
	\max\Big(u_t+H_1(x,t,u,Du),u_t+H_2(x,t,u,Du)\Big)\geq0 \;,\\
	\min\Big(u_t+H_1(x,t,u,Du),u_t+H_2(x,t,u,Du)\Big)\leq0 \;.\\
	\end{cases}
        \end{equation*}
        The notion is ``classical'' in the sense that testing is done with
        test-functions in $C^1(\R^N\times[0,\Tf])$ contrary to (FLS) and (JVS)
        below.
\end{app}

For the ``network approach'', we use the following space of test-functions

\begin{app}{}{Test-functions in the ``network approach''}{p.~\pageref{not:PCun}}\\
    $\PC1$ is the space of piecewise $C^1$-functions $\psi \in C(\R^N \times [0,\Tf])$
    such that there exist $\psi_1 \in C^1(\Omegb_1\times [0,\Tf])$, $\psi_2 \in C^1(\bar
    \Omega_2\times [0,\Tf])$ such that $\psi = \psi_1$ in $\Omegb_1\times [0,\Tf]$ and $\psi =
    \psi_2$ in $\Omegb_2\times [0,\Tf]$.
\end{app}

Below, we use the convention (since the test-functions are not necessarily smooth
on $\H\times [0,\Tf]$) to use the derivatives of $\psi_1$ in the $H_1$-inequalities and those of $\psi_2$ in the $H_2$-inequalities.

\begin{app}{(FLS)}{Flux-Limited Solution,}{p.~\pageref{defiFL}}\\[2mm]
  \noindent A locally bounded function $u: \R^N \times (0,\Tf) \rightarrow \R$ is a
            flux-limited subsolution of \HJgen-\FL (\FLSub in short) if it is a classical viscosity subsolution of \HJgen and if, for
            any test-function $\psi \in \PC1$ and any local maximum point $(x,t) \in \H\times
            (0,\Tf) $ of $u^*-\psi$ in $\R^N\times (0,\Tf)$, at $(x,t)$ the following inequality
            holds 
            \[
              \max \Big(\psi_t + G(x,t,u^*,D_\H \psi), \psi_t+ H^+_1(x,t, u^*, D\psi_1) ,  
              \psi_t+ H^-_2(x,t, u^*, D\psi_2) \Big) \leq 0  \; ,
            \]
            where $u^*=u^*(x,t)$.

            \noindent  A locally bounded function $v: \R^N \times (0,\Tf) \rightarrow \R$
                is a flux-limited supersolution of \HJgen-\FL (\FLSup in short)  if it is a classical viscosity supersolution of \HJgen and
                if, for any test-function $\psi \in \PC1$ and any local minimum point $(x,t) \in
                \H\times (0,\Tf) $ of $v_*-\psi$ in $\R^N\times (0,\Tf)$, at $(x,t)$ the following
                inequality holds
            \[
              \max \Big(\psi_t + G(x,t,v_*,D_\H \psi), \psi_t+ H^+_1(x,t, v_*, D\psi_1) ,  
              \psi_t+ H^-_2(x,t, v_*, D\psi_2) \Big) \geq 0  \; ,
            \]
            where $v_*=v_*(x,t)$.\\
 \noindent A locally bounded function is a flux-limited solution if it is both a \FLSub and a \FLSup.
 \end{app}

\begin{app}{(JVS)}{Junction Viscosity solution,}{p.~\pageref{defiJVS}}\\[2mm]
A locally bounded function $u: \R^N \times (0,\Tf) \rightarrow \R$ is a \JVSub of
    \HJgen-\GJC if it is a classical viscosity subsolution of \HJgen and if, for any test-function
    $\psi=( \psi_1, \psi_2) \in \PC1$ and any local maximum point $(x,t) \in \H\times (0,\Tf) $ of
    $u^*-\psi$ in $\R^N\times (0,\Tf)$,
$$
        \min \Big(G(x,t,\psi_t,D_\H \psi,\frac{\partial \psi_1}{\partial n_1},
        \frac{\partial \psi_2}{\partial n_2}),  \psi_t+ H_1(x,t, u^*, D\psi_1) ,  
        \psi_t+ H_2(x,t, u^*, D\psi_2) \Big) \leq 0  \: ,
$$
    where $u^*$ and the derivatives of $\psi,\psi_1,\psi_2$ are taken at $(x,t)$.\\[2mm]
    A locally bounded function $v: \R^N \times (0,\Tf) \rightarrow \R$ is a \JVSup of
    \HJgen-\GJC if it is a classical viscosity supersolution of \HJgen and if, for any test-function
    $\psi=( \psi_1, \psi_2) \in \PC1$ and any local minimum point $(x,t) \in \H\times (0,\Tf) $ of
    $v_*-\psi$ in $\R^N\times (0,\Tf)$,
$$
        \max \Big(G(x,t,\psi_t,D_\H \psi,\frac{\partial \psi_1}{\partial n_1},
        \frac{\partial \psi_2}{\partial n_2}),  \psi_t+ H_1(x,t, v_*, D\psi_1) ,  
        \psi_t+ H_2(x,t, v_*, D\psi_2) \Big) \geq 0  \: ,
$$
    where $v_*$ and the derivatives of $\psi,\psi_1,\psi_2$ are taken at $(x,t)$.\\[2mm]
    A \JVS (\ie a junction viscosity solution) is a locally bounded function which is both \JVSub and
\JVSup.
\end{app}

\begin{app}{(S-Sub) / (S-Super)}{Stratified sub/supersolutions,}{p.~\pageref{defiStrat}}\\[2mm]
 {\bf 1.\,---} \SSup {\rm :} A locally bounded function $v : \R^N\times [0,\Tf[ \to \R$ is a
    stratified supersolution of \HJBS if $v$---or equivalently $v_*$---is an Ishii supersolution of 
    \eqref{eq:super.H.strat}. 
    
    \medskip

    \noindent {\bf 2.\,---} \wSSub {\rm :}
    A locally bounded function $u : \R^N\times [0,\Tf[ \to \R$ is a {\em weak} stratified subsolution
    of \HJBS if 
    \begin{enumerate}
    \item[$(a)$] for any $k=0,...,(N+1)$, $u^*$ is a viscosity subsolution of 
        $$ \F^k\big(x,t,u^*,Du^*\big) \leq 0 \quad \hbox{on  }\Man{k},
        $$
     \item[$(b)$] similarly, for $t=0$, and $k=0..N$, $u^*(x,0)$ is a viscosity subsolution of
            $$ \F_{init}^k(x,u^*(x,0),D_x u^*(x,0)) \leq 0 \quad \hbox{on  }\Man{k}_0\;.$$
    \end{enumerate}

    \noindent {\bf 3.\,---} \sSSub {\rm :} 
    A locally bounded function $u : \R^N\times [0,\Tf[ \to \R$ is a {\em strong}
    stratified subsolution of \HJBS if it is a \wSSub and satisfies additionally
    \begin{enumerate}
        \item[$(a)$] $\F_*\big(x,t,u^*,Du^*\big)\leq0\quad \hbox{in }\R^N\times(0,\Tf$)\;,
        \item[$(b)$] $(\F_{init})_*(x,u^*(x,0),D_x u^*(x,0)) \leq 0 \quad \hbox{in  }\R^N\;.$
    \end{enumerate}

    \noindent {\bf 4.\,---}  A weak or strong \emph{stratified solution} is a
function which is both a \SSup and either a \wSSub or a \sSSub.
\end{app}

\section*{``Good Assumptions'' for the Network Approach}

\begin{app}{\HBAHJp}{Nonnegative discount factor -- Hamiltonian version,}{p.~\pageref{page:HBAHJp}}\\[2mm]
    This is assumption \HBAHJ in which we assume $\gamma(R) \geq 0$ for any $R$.
    \end{app}

\begin{app}{\HBACPp}{Nonnegative discount factor -- Control version,}{p.~\pageref{page:HBAHJp}}\\[2mm]
    This is assumption \HBACP in which we assume $c(x,t,\alpha) \geq 0$ for any $x,t,\alpha$.
\end{app}

\begin{app}{\NCHJ}{Normal controllability for general Hamiltonians,}{p.~\pageref{page:NCHJ}}\\[2mm]
For any $R>0$, there exists constants $C^R_2,C^R_3,C^R_4>0$ such that, for any $(x,t)\in \H\times
(0,\Tf)$ with $|x|\leq R$, $|u|\leq R$ and $p=(p',p_N)$ with $p'\in \R^{N-1}$, $p_N\in \R$, 
$$H(x,t,u,p)\geq C^R_2 |p_N| - C^R_3|p'| -C^R_4\;.$$
\end{app}

\begin{app}{\TCsH}{Tangential Continuity for general Hamiltonians,}{p.~\pageref{page:TCsH}}\\[2mm]
for any $R>0$, there exists $C^R_1>0$ and a modulus of continuity $m^R : [0,+\infty[ \to
[0,+\infty[$ such that for any $x=(x',x_N), y=(y',x_N)$ with $|x|, |y|\leq R$, $|x_N|\leq R^{-1}$,
$t,s\in [0,\Tf]$, $|u|\leq R$, $p=(p',p_N)\in \R^{N}$,
$$|H(x,t,u,p)- H(y,s,u,p)|\leq C^R_1(|x'-y'|+|t-s|)|p'| +m^R\big(|x'-y'|+|t-s|\big)\;.$$
\end{app}

With these assumptions we can formulate several ``good assumptions'' on $H_1,H_2$ depending on the context\\

\begin{app}{\GAGen}{General case,}{p.~\pageref{page:GAGen}}\ $H_1,H_2$ satisfy \HBAHJp and \NCHJ.
    \end{app}

\begin{app}{\GAConv}{Convex case,}{p.~\pageref{page:GAConv}}\ 
    $H_1,H_2$ satisfy \GAGen, and convex in $p$.
    \end{app}

\begin{app}{\GAQC}{Quasi-convex case,}{p.~\pageref{page:GAQC}}\ $H_1,H_2$ satisfy \GAGen and \HQC.
    \end{app}

\begin{app}{\GACC}{Control case,}{p.~\pageref{page:GACC}}\ \HBACPp and \NCoH are satisfied.
\end{app}

\bigskip

\noindent\textbf{Assumptions on the junction condition $G$} 

\begin{app}{\HContG}{A general assumption,}{p.~\pageref{page:HContG}}\\[2mm]
    For any $R>0$, there exist constants $C^R_5, C^R_6$ such that, for any $x,y \in \H$, $t,s\in
    [0,\Tf]$, $|r|\leq R$, $p'_1,p'_2\in \R^{N-1}$, $a,b,c, a',b',c' \in \R$ 
    $$ |G(x,t,a, p_1',b,c)-G(y,s,a,p'_1,b,c)| \leq C^R_5 (|x-y|+|t-s|)\big(1+ |p_1'| +
    \e_0(|a|+|b|+|c|)\big)\; .$$ $$ |G(x,t,a', p_2',b',c')-G(x,t,a,p_1',b,c)| \leq C^R_6
    (|p_2'-p'_1| + (|a'-a|+|b'-b|+|c'-c|)\big)\; .$$
\end{app}

With this assumption we can formulate several ``good assumptions'' on $G$ depending on the context:

\begin{app}{\GAGFL}{Flux-Limited,}{p.~\pageref{page:GAGFL}}\\[2mm] 
    $G$ is independent of $a,b,c$ and \HContG holds with $\e_0=0$.
\end{app}

\begin{app}{\GAGKT}{Kirchhoff type,}{p.~\pageref{page:GAGKT}}\\[2mm]
    \HContG  holds with $\e_0=0$ and there exists $\alpha \geq 0$, $\beta>0$ such that, for any $x\in \H$, $t\in (0,\Tf)$, $r_1\geq r_2$,
$p'\in \H$, $a_1\geq a_2$, $b_1\geq b_2$, $c_1\geq c_2$, 
\begin{equation}\label{basic-hyp-GJC-A}
    \begin{aligned}
        G(x,t,r_1,a_1,p',b_1,c_1)\,-\, & G(x,t,r_2, a_2,p',b_2,c_2) \\[1mm] & \geq \alpha(a_1-a_2) + 
    \beta(b_1-b_2)+ \beta(c_1-c_2)\;.
    \end{aligned}
\end{equation}
\end{app}

\begin{app}{\GAGFLT}{Flux-limited type,}{p.~\pageref{page:GAGFLT}}\\[2mm] 
    $G(x,t,a, p',b,c)=G_1(a, p',b,c) +G_2(x,t,a, p')$ where $G_1$ is a Lipschitz continuous function
    which satisfies \eqref{basic-hyp-GJC-A} with $\alpha > 0$, $\beta=0$ while $G_2$ satisfies \GAGFL.
\end{app}

\section*{``Good Assumptions'' for Stratified Problems in $\R^{N}\times
       [0,\Tf)$}

\begin{app}{\HBASF}{Basic Assumptions on the Stratified Framework,}{p.~\pageref{page:HBASF}}
\begin{enumerate}
    \item[$(i)$] There exists a \TFS $\M=(\Man{k})_{k=0...(N+1)}$ of $\R^{N}\times
        (0,\Tf)$ such that, for any $r\in \R,p\in \R^{N+1}$, $(x,t)\mapsto\F(x,t,r,p)$ is continuous on $\Man{N+1}$ and may be
        discontinuous on $\Man{0} \cup \Man{1}\cup \cdots \cup \Man{N}$. Moreover $(0_{\R^N},1)
        \notin (T_{(x,t)}\Man{k})^\bot$ for any $(x,t) \in \Man{k}$ and for any
        $k=1...N$~\footnote{This assumption, whose aim is to avoid ``flat part'' of $\Man{k}$ in
        time, will be redundant to the normal controllability assumption in $\R^N\times(0,\Tf)$.  }.
        In the same way, there exists a \TFS $\M_0=(\Man{k}_0)_{k=0...N}$ of $\R^N$ such
        that, for any $r\in \R,p_x\in \R^{N}$, the Hamiltonian $x\mapsto\F_{init}(x,r,p_x)$ is continuous on $\Man{N}_0$ and may be
        discontinuous on $\Man{0}_0 \cup \Man{1}_0 \cup \cdots \cup
        \Man{N-1}_0$.
    \item[$(ii)$] The ``good framework for HJB Equations with discontinuities'' holds for
        Equation~\eqref{eq:super.H.strat} in $\OO=\R^N\times (0,\Tf)$ associated to the
        stratification $\M$.
    \item[$(iii)$] The ``good framework for HJB Equations with discontinuities'' holds for the
        equation $\F_{init}=0$ in $\OO=\R^N$, associated to the stratification $\M_0$. 
\end{enumerate}
\end{app}

We recall that the assumptions for a ``Good Framework for HJB Equations with Discontinuities'' are
that \HBCL, \TCBCL (p.~\pageref{page:TCBCL}) and \NCBCL (p.~\pageref{page:NCBCL}) hold. We refer to
Section~\ref{sec:GFHJD} where the connections with Hamiltonian assumptions \Mong, \TC, \NCe are
described.

\begin{app}{\HSBJ}{Assumptions for the Stratified Barron-Jensen framework,}{p.~\pageref{page:HSBJ}}
\begin{enumerate}
    \item[$(i)$] The stratification does not depend on time: for any $k=0..N$,
        $$\Man{k+1} = \tMan{k} \times \R\;, $$
        where $(\tMan{k})_k$ is a stratification of $\R^N$.
    \item[$(ii)$] We are given a classical \lsc and bounded initial data $g$, \ie we assume the 
        \lsc sub and supersolutions $u$ and $v$ we are considering satisfy
        $$ u(x,0) \leq g(x) \leq v(x,0) \quad \hbox{in  }\R^N\;.$$
    \item[$(iii)$] Hamiltonian $\F$ is a classical Hamiltonian of the form
        $$ \F\big(x,t,r,(p_t,p_x)\big)= p_t + \tilde \F(x,t,r,p_x)\; ,$$
        and there exists $0<\tTf \leq \Tf$ such that $\tilde \F$ is independent of $t$ if
        $0\leq t\leq \tTf$ and coercive, \ie there exists $\nu >0$ such that
        $$ \tilde \F(x,t,r,p_x) = \tilde \F(x,\tTf,r,p_x)\geq \nu |p_x| -M|r| -M\; ,$$
        for any $x\in \R^N$, $t\in [0,\tTf]$, $r\in \R$ and $p_x\in \R^N$,
        $M$ being the constant appearing in the assumptions for $\BCL$\;.
    \item[$(iv)$] The ``good framework for stratified solutions'' is satisfied.
\end{enumerate}
\end{app}

\section*{``Good Assumptions'' for Stratified Problems in the State-Constraints Case}

\begin{enumerate}
\item We still denote by \HBASF the case when this assumption is satisfied with $\R^N$ replaced by $\Omegb$.
\item \HBASFstar means that only Hypotheses \HBASF-$(i)$-$(ii)$ hold, again with $\R^N$ replaced by $\Omegb$. 
\end{enumerate}
\ 

\begin{app}{\HBAIDCP}{Basic Assumption on the Initial Data for the Cauchy
    Problem,}{p.~\pageref{page:HBAIDCP}}\\[2mm]
    There exists $u_0 \in C(\Omegb)$ such that
    $$\begin{cases} \F_{init}(x,r,p)=r-u_0(x) & \text{in } \Omegb \times \{0\}\;,\\
      \F^k_{init}(x,r,p)=r-u_0(x) & \text{on }\domeg \times \{0\}\;,\quad k=0..N\;.
    \end{cases}$$
\end{app}

\begin{app}{\QRB}{Quasi-regular boundary assumption,}{p.~\pageref{page:QRB}}\\[2mm]
For any $(x,t) \in \domeg \times (0,\Tf)$, if $(x,t) \in \Man{k}$, then there exists $r_0>0$ such that
$$ \left[ \Man{k+1}\cup \cdots \Man{N+1}\right] \cap B((x,t),r_0) \; \hbox{is connected}\; .$$
\end{app}

\begin{app}{\IDP}{Inward-pointing Dynamic Property,}{p.~\pageref{page:IDP}}\\[2mm]
    For any $x\in \domeg$,
    there exists $\tau,r>0$ and a $C^1$-function $\phi$ defined in $B(x,r)$ such that $\phi(y) = 0$
    if $y \in \domeg\cap B(x,r)$, $\phi(y)>0$ if $y \in \Omega\cap B(x,r)$, satisfying $$\hbox{For
    all  }(y,s)\in \big(\Omega\cap B(x,r)\big)\times[0,\tau]\;,\quad \sup_{\alpha \in A}\left\{
b(y,s,\alpha)\cdot D_x\phi (y)\right\} \geq 0\;.$$ 
\end{app}

\begin{app}{\IDPN}{Inward-pointing Dynamic Property for a Network,}{p.~\pageref{page:IDPN}}\\[2mm]
     For any $1\leq k \leq N$ and for each connected component $\Man{k}_i$ of $\Man{k}$, the
    assumptions of Lemma~\ref{lem:sufficient.cone} hold true for any $(y,t) \in
    \overline{\Man{k}_i}$ with $\Omega$ replaced by $\tMan{k}_i$ seen as a domain in $\R^{k-1}$.
\end{app}

\begin{app}{\HSBC}{Simplified Framework for Classical Boundary
    Conditions,}{p.~\pageref{page:simplifications}}\\[2mm]
    This set of assumptions essentially means that $(i)$ this is a standard HJB problem with
    Cauchy initial data; $(ii)$ the equation has no discontinuities inside the domain; $(iii)$ the
    domain is bounded, associated to a time-independent stratification. Moreover, the ``good
    framework'' holds, namely \HBACP and \HBASFstar hold.
\end{app}

\begin{app}{\Hgam[\omega]}{Natural Assumptions on $\gamma$ and $g$ on
    $\omega$,}{p.~\pageref{page:Hgam}}
    \begin{enumerate}
    \item[$(i)$] There exists $\nu>0$ and a Lipschitz continuous $\gamma_\omega:
            \R^N\times\R \to \R^N$ such that $\gamma=\gamma_\omega$ on $\omega\times[0,\Tf]$ and
            \begin{equation}\label{assump:NP-A} \gamma_\omega (x,t)\cdot n(x)\geq \nu >0\quad \hbox{on
            } \omega\times [0,\Tf], \end{equation} where $n(x)$ is the unit outward normal to $\domeg$
            at $x$\footnote{We point out that, when this assumption holds, $\domeg$ and $\omega$ coincide in a neighborhood of
            each $x\in \omega$ and therefore $\domeg$ is smooth at such points as a consequence of
the assumptions on $\omega$.}. 
\item[$(ii)$] There exists a continuous function $g_\omega:
\R^N\times \R\to \R$ such that $g=g_\omega$ on $\omega\times [0,\Tf]$.  \end{enumerate}\end{app}

\begin{app}{\Hgamma}{Specific Hypotheses for the Oblique Derivative
    Problem,}{p.~\pageref{page:Hgams}}\\[2mm]
    For any $i\in I^N$, \Hgam[\omega] holds for $\omega=\tMan{N}_i$ and we denote by $\gamma_i,g_i$
    the corresponding functions $\gamma_\omega,g_\omega$.  
\end{app}



\fancyhead[CO]{HJ-Equations with Discontinuities: Bibliography}
\bibliographystyle{plain}
\fancyhead[CO]{HJ-Equations with Discontinuities: Bibliography}


\addcontentsline{toc}{part}{Index}
\printindex

\end{document}